\documentclass[12pt]{article}
\usepackage{amsmath}
\usepackage{graphicx}
\usepackage{natbib}
\usepackage{url} % not crucial - just used below for the URL 

\newcommand{\blind}{0}

\usepackage{verbatim,epsfig,amsthm}
\usepackage{amssymb,amsmath,rotating}
\usepackage[nodisplayskipstretch]{setspace}
\usepackage{amsfonts}
\usepackage{multirow}
\usepackage{epsfig,latexsym,enumitem}
\usepackage{mathtools,color,etoolbox,hyperref,float,mathrsfs,color}
\usepackage{caption,subcaption}
\usepackage{tabu}
\usepackage{threeparttable}
\usepackage{makecell}
\usepackage{multirow}
\usepackage{bm}
\usepackage{pdflscape}
\usepackage{booktabs}

 \usepackage{adjustbox,cases,subcaption}

 \usepackage{xr} 

\newcommand{\bd}{\boldsymbol}
\newcommand{\mb}{\mathbf}

\newcommand{\be}{\begin{equation}}
\newcommand{\ee}{\end{equation}}

\DeclareMathOperator*{\argmin}{arg\,min}

\DeclareMathOperator*{\argsort}{arg\,sort}
\DeclareMathOperator*{\HBIC}{HBIC}
\DeclareMathOperator*{\BIC}{BIC}

\DeclareMathOperator{\diag}{diag}

\DeclareMathOperator{\cov}{cov}
\DeclareMathOperator{\corr}{corr}
\DeclareMathOperator{\var}{var}

\DeclareMathOperator{\lspan}{span}

\DeclareMathOperator{\sd}{sd}

\DeclareMathOperator{\supp}{supp}
\DeclareMathOperator{\sign}{sgn}
\DeclareMathOperator{\AvgLength}{AvgLength}
\DeclareMathOperator{\TrueLength}{TrueLength}
\DeclareMathOperator{\CovRate}{CovRate}
\DeclareMathOperator{\AbsAvgZ}{AbsAvgZ}
\DeclareMathOperator{\SDZ}{SDZ}
\DeclareMathOperator{\SDAR}{SDAR}

\newtheorem{theorem}{Theorem}
\newtheorem{lem}{Lemma}
\newtheorem{cor}{Corollary}

\newtheorem{remark}{Remark}

\usepackage{algorithm}
\usepackage{algpseudocode}
\makeatletter
\renewcommand{\fnum@algorithm}{\fname@algorithm~\thealgorithm.}
\makeatother

\algnewcommand\INPUT{\item[\textbf{Input:}]}%
\algnewcommand\OUTPUT{\item[\textbf{Output:}]}%

\begin{document}

\def\spacingset#1{\renewcommand{\baselinestretch}%
{#1}\small\normalsize} \spacingset{1}

%%%%%%%%%%%%%%%%%%%%%%%%%%%%%%%%%%%%%%%%%%%%%%%%%%%%%%%%%%%%%%%%%%%%%%%%%%%%%%

\if0\blind
{
  \title{\bf Nodewise Loreg:     Nodewise $L_0$-penalized Regression for  High-dimensional Sparse Precision Matrix Estimation}
	\author{Hai Shu$^{1,\dag}$, Ziqi Chen$^{2,\dag,*}$, Yingjie Zhang$^{2,\dag}$, and Hongtu Zhu$^{3,*}$
\hspace{.2cm}\\
\\
		$^1$Department of Biostatistics, School of Global Public Health, New York University\\
		$^2$School of Statistics, KLATASDS-MOE, East China Normal University\\ 
		$^3$Departments of Biostatistics, Statistics, Computer Science, and Genetics,\\
The University of North Carolina at Chapel Hill\\
$^\dag$Equal contributions\\
		$^*$Correspondence: zqchen@fem.ecnu.edu.cn, htzhu@email.unc.edu
		}
  \maketitle
} \fi

\if1\blind
{
  \bigskip
  \bigskip
  \bigskip
  \begin{center}
    {\LARGE\bf Title}
\end{center}
  \medskip
} \fi

\bigskip
\begin{abstract}
We propose Nodewise Loreg, a nodewise $L_0$-penalized regression method for estimating  high-dimensional sparse precision matrices. We establish its asymptotic properties, including convergence rates, support recovery, and asymptotic normality under high-dimensional sub-Gaussian settings. Notably, the Nodewise Loreg estimator is asymptotically unbiased and normally distributed, eliminating the need for debiasing required by Nodewise Lasso.
  We also develop a desparsified version of Nodewise Loreg, similar  to the desparsified Nodewise Lasso estimator. The asymptotic variances of the undesparsified Nodewise Loreg estimator are upper bounded by those of both desparsified Nodewise Loreg and Lasso estimators for Gaussian data, potentially offering more powerful statistical inference. Extensive simulations show that the undesparsified Nodewise Loreg estimator generally outperforms the two desparsified estimators in asymptotic normal behavior. Moreover, Nodewise Loreg surpasses Nodewise Lasso, CLIME, and GLasso in most simulations in terms of matrix norm losses, support recovery, and timing performance. 
Application to a breast cancer gene expression dataset further demonstrates Nodewise Loreg's superiority over the three $L_1$-norm based methods.

\end{abstract}

\noindent%
\vfill

\newpage
\spacingset{1.9} % DON'T change the spacing!

\section{Introduction}
\label{s:intro}

The precision matrix, or inverse covariance matrix, is fundamental in multivariate analysis and applications such as linear and quadratic discriminant analysis \citep{MR2722294}, generalized least squares \citep{MR2814522}, Markowitz portfolio optimization \citep{MR3837524}, and minimum variance beamforming \citep{Vorobyov2013}. It is particularly crucial in undirected graphical models, which represent the conditional dependence structure among variables \citep{MR1419991}.  
Let $\bd{x}=(x_1,\dots, x_p)^\top\in \mathbb{R}^p$ be a vector  of $p$ random variables
with a precision matrix $\mb{\Omega}\in \mathbb{R}^{p\times p}$.
In an undirected graphical model,  
$\bd{x}$ is represented as $p$ nodes in a graph, where an undirected edge between 
$x_i$ and $x_j$ is drawn if they are conditionally dependent given all other variables.
This conditional dependence corresponds to a
nonzero entry $(i,j)$ in $\mb{\Omega}$, when $\bd{x}$ is a Gaussian, or more broadly, a nonparanormal random vector \citep{Liu09}. 
Undirected graphical models have been widely used in fields such as brain functional connectivity analysis \citep{ryali2012estimation}, genetic network analysis \citep{Schafer05}, and social network analysis \citep{farasat2015probabilistic}, among others \citep{giudici2016graphical, epskamp2018gaussian}. 
%it reveals conditional independence and is more relevant to explore direct connectivity.

Consider an $n\times p$ data matrix $\mb{X}$, whose rows  are $n$ i.i.d. copies of the
$p$-dimensional real-valued random vector
$\bd{x}$
with a covariance matrix $\mb{\Sigma}=\mb{\Omega}^{-1}$.
 It is well known that the sample covariance matrix derived from $\mb{X}$
is not a consistent estimator of $\mb{\Sigma}$
in high-dimensional settings
where
$p$ grows with $n$ \citep{Bai93,Bai10},
and moreover, it fails to be invertible when $p>n$. 
Different structural assumptions have been imposed to ensure consistent estimation 
of high-dimensional precision matrix $\mb{\Omega}$ \citep{fan2016overview}.
In this paper,
we focus on the sparsity assumption regarding  the number of nonzero entries in $\mb{\Omega}$.

There is a rich literature on the estimation of high-dimensional sparse precision matrix \(\mathbf{\Omega}\). The \(L_1\)-penalized maximum likelihood estimation (MLE) method, known as graphical Lasso (GLasso) \citep{Frie08, Yuan07, Bane08, Roth08, Ravi11, Hsie14}, and its variants with different penalties \citep{Lam09, marjanovic2015, van2016ridge, kuismin2017precision, kovacs2021graphical} have been extensively developed and investigated for their theoretical properties. 
Methods for estimating \(\mathbf{\Omega}\) column-by-column, which can be implemented with parallel computing, include various nodewise regression approaches \citep{Mein06, Yuan10, sun2013sparse, VGeer14, Zhao15, pmlr-v51-wang16a, TigerLiu17, Jank17, klaassen2023uniform}, and constrained \(L_1\)-minimization for inverse matrix estimation (CLIME) \citep{CaiClime} and its variants \citep{liu2015fast, CLZ16}. 
Notably, Nodewise Lasso \citep{Mein06, VGeer14, Jank17}, a nodewise regression method with \(L_1\) penalty, converts the large precision matrix estimation into a high-dimensional sparse linear regression problem, where each node \(x_j\) is regressed on the remaining nodes \(\{x_i\}_{i \ne j}\) using the Lasso method \citep{Lasso}.

High-dimensional sparse linear regression has been extensively studied. The ideal penalty for promoting sparsity penalizes the \(L_0\) norm of the regression coefficient vector, which counts the number of nonzero coefficients. Unfortunately, the \(L_0\) penalty is discontinuous and nonconvex, making the optimization problem NP-hard \citep{MR1320206,chen2014complexity}. 
The popular Lasso method \citep{Lasso} uses the \(L_1\) norm instead, resulting in a convex minimization problem. However, the \(L_1\) penalty tends to overshrink large coefficients, producing biased estimates \citep{fan2001variable}. To reduce this bias, alternatives like the adaptive Lasso \citep{adaptiveLasso}, SCAD \citep{fan2001variable}, and MCP \citep{MCP2010} 
penalties have been proposed as relaxed surrogates of the \(L_0\) penalty. 
Several algorithms approximate \(L_0\)-penalized linear regression, including GraDes \citep{garg2009gradient}, FoBa \citep{zhang2011adaptive}, OMP \citep{zhang2011sparse}, MIO \citep{MR3476618}, ABESS \citep{zhu2020polynomial}, and SDAR \citep{Huang18}. SDAR stands out for its computational efficiency, with complexity \(O(np)\) per iteration, and its robust statistical properties such as consistency and sparsistency. Extensive simulations have demonstrated the superior performance of the \(L_0\) penalty over non-\(L_0\) penalties \citep{johnson2015risk, Huang18, dai2023variable}.

Another stream of work focuses on ``de-biasing" (or ``de-sparsifying") procedures for Lasso-type estimators in high-dimensional linear regression. These procedures yield estimators of regression coefficients that are asymptotically unbiased and normally distributed, enabling statistical inference for high-dimensional linear models \citep{zhang2014confidence, VGeer14, javanmard2014confidence, javanmard2014hypothesis, 10.1214/17-AOS1630, 10.1111/ectj.12097}.
Building on these advancements, researchers have developed debiased (desparsified) Nodewise Lasso-type estimators for large precision matrices \citep{Jank17, chang2018confidence, klaassen2023uniform}, which are asymptotically unbiased and enjoy asymptotic normality for inference on matrix parameters. 
Instead of estimating the regression coefficients, \citet{Zhao15} estimate the noise level using scaled Lasso \citep{sun2012scaled}, achieving asymptotic normality without debiasing but requiring post-thresholding to provide a sparse estimator of the precision matrix. 
Inspired by the debiased Lasso \citep{VGeer14}, \citet{Jank15} derive a desparsified GLasso estimator by inverting the Karush-Kuhn-Tucker (KKT) conditions of the GLasso optimization problem, proving the estimator's asymptotic normality.

We propose Nodewise Loreg, a novel {\bf nodewise} $\mathbf{L_0}$-penalized {\bf reg}ression method for estimating large sparse precision matrices, in contrast to  Nodewise Lasso that uses the \(L_1\) penalty \citep{Mein06, VGeer14, Jank17}. Utilizing the SDAR algorithm \citep{Huang18} for \(L_0\)-penalized regression, we establish comprehensive asymptotic results for Nodewise Loreg, including convergence rates, support recovery, and asymptotic normality in high-dimensional sub-Gaussian settings. Notably, the Nodewise Loreg estimator is asymptotically unbiased and normally distributed, without the need for debiasing.
Moreover, its asymptotic normality has  entrywise variances dominated by those of the desparsified Nodewise Lasso
estimator \citep{Jank17} under Gaussian conditions,
potentially providing more powerful statistical inference. 
Additionally, we develop a desparsified Nodewise Loreg estimator with the same  asymptotic normality as the desparsified Nodewise Lasso estimator. Extensive simulations demonstrate that the undesparsified Nodewise Loreg estimator generally outperforms the two desparsified estimators in asymptotic normal behavior. Furthermore, in most simulation settings, Nodewise Loreg surpasses Nodewise Lasso, CLIME, and GLasso in terms of matrix norm losses, support recovery, and timing performance. While related to linear regression, our method requires a more involved treatment due to its nonlinear nature and the adaptation of SDAR theory from fixed to random design settings.

%To the best of our knowledge, our work is the first to theoretically establish the asymptotic theory for 
%the $L_0$-penalized nodewise regression estimator. 
%We notice that \citet{marjanovic2015%}
%and \citet{kim2021scalable} also use% the $L_0$ penalty in different ways
%for large pr%ecision matrix estimation.
%Specifically,
%\citet{marjanovic2015}
%propose a $L_0$-penalized MLE method,
%which replaces the $L_1$ penalty 
%in GLasso
%with
%the $L_0$ penalty.
%They apply a cyclic descent algorithm
%to solve the optimization, but their theoretical analysis is limited to algorithmic convergence and does not address  the asymptotic properties of the estimator.
%\citet{kim2021scalable} adopt the innovated  scalable efficient estimation framework of  \citet{MR3546445}
%to estimate $\mb{\Omega}$, which is also the covariance matrix of 
%$\widetilde{\bd{x}}:=\mb{\Omega}\bd{x}$.
%They
%estimate the data matrix
%of $\widetilde{\bd{x}}$ (i.e., $\mb{\Omega}\mb{X}^\top\in \mathbb{R}^{p\times n}$)
%using $L_0$-penalized linear regression via the SDAR algorithm, and then use a hard-thresholded sample covariance matrix of $\widetilde{\bd{x}}$
%as the estimator of $\mb{\Omega}$,
%without providing 
%any asymptotic theory.

To the best of our knowledge, our work is the first to theoretically establish the asymptotic properties of the \(L_0\)-penalized nodewise regression estimator. Previous works by \citet{marjanovic2015} and \citet{kim2021scalable} also use the \(L_0\) penalty for large precision matrix estimation, but in different ways.
\citet{marjanovic2015} propose an \(L_0\)-penalized MLE method, replacing the \(L_1\) penalty in GLasso with the \(L_0\) penalty, and use a cyclic descent algorithm for optimization. However, their analysis is limited to algorithmic convergence and does not address the asymptotic properties of the estimator.
\citet{kim2021scalable} adopt the innovated scalable efficient estimation framework of \citet{MR3546445} to estimate \(\mathbf{\Omega}\), which is also the covariance matrix of \(\widetilde{\bd{x}} := \mathbf{\Omega}\bd{x}\). They estimate the data matrix \(\mathbf{\Omega}\mathbf{X}^\top \in \mathbb{R}^{p \times n}\) of $\widetilde{\bd{x}}$  using \(L_0\)-penalized linear regression via the SDAR algorithm, then use a hard-thresholded sample covariance matrix of \(\widetilde{\bd{x}}\) as the estimator of \(\mathbf{\Omega}\), but do not provide any asymptotic theory.

%The rest of this paper is organized as follows.%
%Section~\ref{s:method} proposes our Nodewise Lo%reg estimator for high-dimensional sparse precision matrix,
%following an introduction of the $L_0$-penalize%d linear regression and its SDAR algorithm. 
%Section~\ref{s: theoretical resul%ts}
%provide%s
%the asy%mptotic properties of the Nodewise Loreg estimator, including convergence rates, support recovery, and asymptotic normality,
%and als%o discusses a desparisified version of the estimator. 
%We  eva%luate the performance of Nodewise Loreg
%against%  Nodewise Lasso, CLIME, and GLasso 
%through% extensive simulations in Section~\ref{s: simulations}
%and through analysis of the MDA133 breast cancer gene expression dataset \citep{hess2006pharmacogenomic} in Section~\ref{sec: real data}.
%Section~\ref{s: conclusion} concludes the paper and discusses  possible extensions.
%Supplementary Materials provide all theoretical proofs, additional simulation results, and 
%the computer code for simulations and real data~analysis.

The rest of this paper is organized as follows. Section~\ref{s:method} proposes the Nodewise Loreg estimator for high-dimensional sparse precision matrices, following an introduction of the \(L_0\)-penalized linear regression and the SDAR algorithm. Section~\ref{s: theoretical results} presents the asymptotic properties of Nodewise Loreg  and discusses a desparsified version of the estimator. Section~\ref{s: simulations} evaluates the performance of Nodewise Loreg against Nodewise Lasso, CLIME, and GLasso through extensive simulations, and Section~\ref{sec: real data} analyzes the MDA133 breast cancer gene expression dataset \citep{hess2006pharmacogenomic}. Section~\ref{s: conclusion} concludes the paper and suggests possible extensions. Supplementary Materials provide all theoretical proofs, additional simulation results, and the computer code for simulations and real data analysis.

\section{Method}\label{s:method}

\subsection{Notation}

We now introduce some useful notation. 
Define $[n]=\{1,2,\dots, n\}$ for any positive integer $n$.
For any matrix $\mb{M}=(M_{ij})_{n\times p}\in \mathbb{R}^{n\times p}$, we denote the spectral norm $\| \mb{M} \|_2=\lambda_{\max}^{1/2}(\mb{M}^\top \mb{M})$,
where $\lambda_{\max}(\cdot)$ and $\lambda_{\min}(\cdot)$ are, respectively, the largest and smallest eigenvalues of the input square matrix, 
the Frobenius norm 
$\| \mb{M}\|_F=(\sum_{i=1}^n \sum_{j=1}^p M_{ij}^2)^{1/2}$,
the max norm $\|\mb{M} \|_{\max}=\max_{i\in [n], j\in [p]}|M_{ij}|$,
the $L_1$ norm
$\| \mb{M}\|_1=\max_{j\in [p]} \sum_{i=1}^n | M_{ij}|$,
and
the $L_\infty$ norm 
$\| \mb{M}\|_\infty=\max_{i\in [n]} \sum_{j=1}^p | M_{ij}|$.
The support set of $\mb{M}$ is defined
by $\supp(\mb{M})=\{(i,j)\in [n]\times [p]: M_{ij}\ne 0\}$.
Matrix subsetting is defined as
$\mb{M}_{AB}=(M_{ij})_{i\in A, j\in B}$, where
$A\subseteq [n]$ and $B\subseteq [p]$ are the subsets of  row indices and column indices of $\mb{M}$, respectively.
 Special subsetting symbols 
 $*$, $\setminus k$, and $k$ 
 imply complete, all-but-one, and single index selections for rows or columns.
For clarity, we sometimes use a comma to separate $A$ and $B$ in the subscript
of $\mb{M}$ when subsetting the matrix.
For any vector $\bd{v}=(v_1,\dots, v_p)^\top \in \mathbb{R}^p$,
we denote the $L_0$ norm $\|\bd{v}\|_0=\sum_{i=1}^pI(v_i\ne 0)$ 
with  $I(\cdot)$ being the indicator function,
define $\bd{v}_{\setminus j}$ as the vector obtained by excluding $v_j$ from $\bd{v}$,
and let
$\bd{v}_S$ denote
the
vector consisting of the entries of $\bd{v}$ indexed by
a given set $S$.
Let $\Phi(\cdot)$ be the cumulative distribution function of the standard Gaussian distribution.
Throughout the rest of the paper, by default we assume both $n,p\to \infty$   in the asymptotic arguments.

\subsection{$L_0$-penalized linear regression}
We consider the high-dimensional linear regression model:
\[
\bd{y}=\mb{Z}\bd{\beta}^*+\bd{\eta}
\]
where $\bd{y}\in \mathbb{R}^n$ is the response vector,
$\mb{Z}\in \mathbb{R}^{n\times p}$ is the design matrix with
$\sqrt{n}$-normalized columns (i.e., $\|\mb{Z}_{*j}\|_2=\sqrt{n}$, $j\in[p]$),
$\bd{\beta}^*\in\mathbb{R}^p$
is the vector of regression coefficients,
and $\bd{\eta}\in \mathbb{R}^n$ is the vector of random noises.
Assume that $\bd{\beta}^*$ is sparse, that is, the number of its nonzero entries is small  relative to $n,p$.
We can then estimate $\bd{\beta}^*$ by the solution of the $L_0$ minimization problem:
 \be\label{eqn: L0 reg}
 \min_{\bd{\beta}\in\mathbb{R}^p} \| \bd{y} -\mb{Z}\bd{\beta}\|_2^2/n\quad\text{subject to}\quad \|\bd{\beta} \|_0\le s,
 \ee
 where $s>0$ controls the sparsity level. 
Unfortunately, this optimization is known to be an NP-hard problem~\citep{MR1320206,chen2014complexity},
making it computationally intractable for large $p$.
To address this challenge, many
works have turned to relaxed surrogates for the $L_0$ penalty,
such as the widely-used Lasso ($L_1$-norm) penalty~\citep{Lasso}, the adaptive Lasso penalty~\citep{adaptiveLasso},
the SCAD penalty~\citep{fan2001variable}, and the MCP penalty~\citep{MCP2010}.
Alternatively, 
several
algorithms have been proposed to approximately solve the $L_0$-penalized linear regression, 
including  GraDes \citep{garg2009gradient}, FoBa \citep{zhang2011adaptive},  OMP \citep{zhang2011sparse},  MIO \citep{MR3476618}, ABESS \citep{zhu2020polynomial}, and SDAR \citep{Huang18}.

We consider the SDAR algorithm \citep{Huang18}, 
which is notable for its computational efficiency, with a complexity of $O(np)$ per iteration, and its
 comprehensive statistical properties,
 such as consistency in different norms and sparsistency of parameter estimates,
when compared to other competing algorithms.  
The SDAR algorithm is motivated by the necessary KKT conditions for 
the Lagrangian form of \eqref{eqn: L0 reg}.
Specifically, the Lagrangian form is 
 \be\label{eqn: L0 reg Lagrangian}
 \min_{\bd{\beta}\in\mathbb{R}^p} \| \bd{y} -\mb{Z}\bd{\beta}\|_2^2/n+2\lambda \|\bd{\beta} \|_0.
 \ee
 The KKT conditions for the solution of \eqref{eqn: L0 reg Lagrangian} are
\be\label{KKT cond}
\begin{cases}
\bd{\beta}=(\bd{\beta}+\bd{d})I(|\bd{\beta}+\bd{d}|\ge \sqrt{2\lambda}),\\
 \bd{d}=\mb{Z}^\top(\bd{y}-\mb{Z}\bd{\beta})/n.
\end{cases}
\ee
It follows from the KKT conditions that
\be\label{KKT cond2}
\begin{cases}
A=\{i\in[p]: |\bd{\beta}_{\{i\}}+\bd{d}_{\{i\}}|\ge \sqrt{2\lambda}\},\\
\bd{\beta}_A=(\mb{Z}_A^\top\mb{Z}_A)^{-1}\mb{Z}_A^\top\bd{y},
\qquad
\bd{\beta}_{A^c}=\bd{0},\\
\bd{d}_{A^c}=\mb{Z}_{A^c}^\top(\bd{y}-\mb{Z}_A\bd{\beta}_A)/n,
\qquad
\bd{d}_{A}=\bd{0},
\end{cases}
\ee
where $A^c$ is the complement of set $A$.
The SDAR algorithm finds an approximate sequence of solutions to the  KKT conditions in~\eqref{KKT cond} iteratively based on equation  \eqref{KKT cond2} until convergence. In particular,
instead of tuning the parameter $\lambda$, 
SDAR tunes the size of the active set $A$, denoted by $T$. Thus,
$A$ comprises  the indices of the first $T$ largest entries 
in the vector $|\bd{\beta}+\bd{d}|$.
Algorithm~\ref{alg: SDAR} summarizes the steps of SDAR, where $\argsort(\cdot,\downarrow)[1:T]$ yields the set consisting of the indices of the
first $T$ largest entries  in a given vector.

\begin{algorithm}
\caption{$\SDAR(\bd{y},\mb{Z}, T)$ for \eqref{eqn: L0 reg Lagrangian}} \label{alg: SDAR}
\begin{algorithmic}[1]
\INPUT $\bd{y}$, $\mb{Z}$, and $T$.
\State Initialize $\bd{\beta}^0=\bd{0}$, 
\State $\bd{d}^0=\mb{Z}^\top(\bd{y}-\mb{Z}\bd{\beta}^0)/n$, 
\State $A^0=\argsort(|\bd{\beta}^0+\bd{d}^0|, \downarrow)[1:T]$;
\For{$k=0,1,2,\dots$}
    \State $\bd{\beta}_{A^k}^{k+1}=(\mb{Z}_{A^k}^\top \mb{Z}_{A^k})^{-1}\mb{Z}_{A^k}^\top\bd{y}$, 
    \State $\bd{\beta}_{(A^k)^c}^{k+1}=\bd{0}$,
    \State $\bd{d}_{(A^k)^c}^{k+1}=\mb{Z}_{(A^k)^c}^\top(\bd{y}-\mb{Z}_{A^k}\bd{\beta}_{A^k}^{k+1})/n$, 
    \State $\bd{d}_{A^k}^{k+1}=\bd{0}$,
        \State $A^{k+1}=\argsort(|\bd{\beta}^{k+1}+\bd{d}^{k+1}|, \downarrow)[1:T]$;
\State \textbf{if} $A^{k+1} = A^k$ \textbf{then} stop;
\EndFor
\OUTPUT $\widehat{\bd{\beta}}=\bd{\beta}^{k+1}$, $\widehat{A}=A^k$, and $k$.
\end{algorithmic}
\end{algorithm}

\subsection{Nodewise Loreg: nodewise $L_0$-penalized regression}

Our Nodewise Loreg method modifies the Nodewise Lasso  \citep{Mein06, VGeer14,Jank17} by replacing the $L_1$ penalty with the $L_0$ penalty.
We first introduce the nodewise regression. 
Recall that the rows of data matrix $\mb{X}\in \mathbb{R}^{n\times p}$  are $n$ i.i.d. copies of the
random vector
$\bd{x}\in \mathbb{R}^p$ with a precision matrix $\mb{\Omega}=\mb{\Sigma}^{-1}$.
Following the literature \citep{VGeer14, Jank17}, 
we assume $E(\bd{x})=\bd{0}$ to concentrate on 
the precision matrix estimation, 
avoiding the 
impact of the mean estimation by centering the data.
The projection theorem 
and the inverse formula for partitioned matrices~\citep{Bane14}
provide
the steps of
the population-level nodewise regression:
for each $j\in [p]$, we have 
\begin{align} 
\bd{\alpha}^*_j&=\argmin_{\bd{\alpha}_j\in \mathbb{R}^{p-1}} E(x_j-\bd{x}_{\setminus j }^\top \bd{\alpha}_j)^2=\mb{\Sigma}^{-1}_{\setminus j,\setminus j}\mb{\Sigma}_{\setminus j, j},
\label{lm, projection}\\
\mb{\Omega}_{jj}&=(\mb{\Sigma}_{jj}-\mb{\Sigma}_{j,\setminus j} \mb{\Sigma}_{\setminus j,\setminus j}^{-1} \mb{\Sigma}_{\setminus j,j} )^{-1}
\label{Theta_jj formula}\\
&=(\mb{\Sigma}_{jj}-2 \bd{\alpha}_j^{*\top} \mb{\Sigma}_{\setminus j,j}+\bd{\alpha}_j^{*\top} \mb{\Sigma}_{\setminus j,\setminus j} \bd{\alpha}_j^*)^{-1}
=[E(x_j-\bd{x}_{\setminus j }^\top \bd{\alpha}_j)^2]^{-1},
\label{Omega_jj=sigma_j^-2}\\ 
\mb{\Omega}_{\setminus j,j}&=-\mb{\Omega}_{jj}\bd{\alpha}_j^*.
\label{Theta_-jj formula}
\end{align}

Nodewise Lasso estimates  $\mb{\Omega}$ 
column-by-column
using Lasso for the $L_1$-penalized least squares based on data matrix $\mb{X}$:
 for $j\in[p]$,
\begin{align*}
\widehat{\bd{\alpha}}_j&=\argmin_{\bd{\alpha}_j\in \mathbb{R}^{p-1}}\| \mb{X}_{*j}-\mb{X}_{*\setminus j} \bd{\alpha}_j   \|_2^2/n+2\lambda_j \| \bd{\alpha}_j\|_1,\\
\widehat{\sigma}_j^2&=\| \mb{X}_{*j}-\mb{X}_{*\setminus j} \widehat{\bd{\alpha}}_j  \|_2^2/n+\lambda_j \| \widehat{\bd{\alpha}}_j\|_1,\\
\widehat{\mb{\Omega}}_{jj}&=\widehat{\sigma}_j^{-2},
\quad\text{and}\quad
\widehat{\mb{\Omega}}_{\setminus j,j }=-\widehat{\mb{\Omega}}_{jj}\widehat{\bd{\alpha}}_j.
\end{align*}

In contrast to Nodewise Lasso,
our proposed Nodewise Loreg uses the standardized design matrix 
$\mb{Z}_{*\setminus j}$ 
and the $L_0$ penalty instead of $\mb{X}_{*\setminus j}$
and the $L_1$ penalty in the penalized least squares, which we solve using the SDAR algorithm.
Here, $\mb{Z}=\mb{X}\widehat{\mb{\Gamma}}^{-1/2}$ with $\widehat{\mb{\Gamma}}=\diag\{(\widehat{\mb{\Sigma}}_{ii})_{i\in[p]}\}$ and
$(\widehat{\mb{\Sigma}}_{ii})_{i\in [p]}$ being the diagonal of
$\widehat{\mb{\Sigma}}=\mb{X}^\top \mb{X}/n$.
Algorithm~\ref{alg: Nodewise Loreg} outlines  the  estimating procedure for Nodewise Loreg.

\begin{algorithm}
\caption{Nodewise Loreg}\label{alg: Nodewise Loreg}
\begin{algorithmic}[1]
\INPUT $\mb{X}$ and $\{T_j\}_{j=1}^p$.
    \State $\mb{Z}=\mb{X}\widehat{\mb{\Gamma}}^{-1/2}$ with $\widehat{\mb{\Gamma}}=\diag\{(\widehat{\mb{\Sigma}}_{ii})_{i\in[p]}\}$ and
$\widehat{\mb{\Sigma}}=\mb{X}^\top \mb{X}/n$,

\For{$j=1,\dots,p$}
    \State $(\widehat{\bd{\beta}}_j,\widehat{A}_j,k_j)=\SDAR(\mb{X}_{*j},\mb{Z}_{*\setminus j},T_j)$ from Algorithm~\ref{alg: SDAR}
    
    for $\min_{\bd{\beta}_j\in\mathbb{R}^{p-1}}\| \mb{X}_{*j}-\mb{Z}_{*\setminus j} \bd{\beta}_j   \|_2^2/n+2\lambda_j \| \bd{\beta}_j\|_0$,

\State $\widehat{\bd{\alpha}}_j=\widehat{\mb{\Gamma}}_{\setminus j,\setminus j}^{-1/2} \widehat{\bd{\beta}}_j$,
\State $\widehat{\sigma}_j^2=\|\mb{X}_{*j}-\mb{X}_{*\setminus j} \widehat{\bd{\alpha}}_j  \|_2^2/n$,
\State $\widehat{\mb{\Omega}}_{jj}=\widehat{\sigma}_j^{-2}$,
\State $\widehat{\mb{\Omega}}_{\setminus j,j }=-\widehat{\mb{\Omega}}_{jj}\widehat{\bd{\alpha}}_j$,
\EndFor
\OUTPUT $\widehat{\mb{\Omega}}$.
\end{algorithmic}
\end{algorithm}

For Algorithm~\ref{alg: Nodewise Loreg},
we sometimes write $k_j=k_j(T_j)$ and $\widehat{A}_j=\widehat{A}_j(k_j(T_j))=\widehat{A}_j(T_j)$
to explicitly  indicate their dependence on the tuning parameter $T_j$.
To select the optimal value of $T_j$,
we adopt 
the high-dimensional Bayesian information criterion
(HBIC) \citep{Wang13HBIC,Huang18,zhu2020polynomial}.
For any active set $A_j$ of $\bd{\beta}_j$,
the HBIC is defined as
\be\label{HBIC general}
\HBIC(A_j)=n \log \mathcal{L}_{A_j}+|A_j|\log (p-1)\log(\log n),
\ee
where $\mathcal{L}_{A_j}=\min_{(\bd{\beta}_j)_{A_j^c}=\bd{0}}\mathcal{L}_{n,j}(\bd{\beta}_j)$
and $\mathcal{L}_{n,j}(\bd{\beta}_j)=\| \mb{X}_{*j}-\mb{Z}_{*\setminus j} \bd{\beta}_j   \|_2^2/n$.
Define the optimal value of $T_j$ as
\be\label{HBIC: Tj*}
T_j^*=\min_{0\le T_j\le T_{\max,j}} \HBIC(\widehat{A}_j(T_j)),
\ee
where $T_{\max,j}$ is the largest candidate value of $T_j$.

   Throughout the paper, due to their column-by-column estimating nature, 
   we apply the minimum symmetrization  \citep{CaiClime} to both Nodewise Loreg and Nodewise Lasso:
  \[  
    \widehat{\mb{\Omega}}_{ij}^{\text{S}}=\widehat{\mb{\Omega}}_{ij}^{\text{US}}I(|\widehat{\mb{\Omega}}_{ij}^{\text{US}}|\le |\widehat{\mb{\Omega}}_{ji}^{\text{US}}|)+\widehat{\mb{\Omega}}_{ji}^{\text{US}}I(|\widehat{\mb{\Omega}}_{ij}^{\text{US}}|> |\widehat{\mb{\Omega}}_{ji}^{\text{US}}|),
    \] 
    where $\widehat{\mb{\Omega}}^{\text{S}}$ and $\widehat{\mb{\Omega}}^{\text{US}}$ are the symmetrized and unsymmetrized estimators of $\mb{\Omega}$,  respectively.

\section{Theoretical results}\label{s: theoretical results}

In this section, we study the asymptotic properties of the proposed Nodewise Loreg estimator, including 
convergence rates, support recovery, and asymptotic normality.
We also develop and discuss a desparsified version of the Nodewise Loreg estimator, analogous  to the desparsified Nodewise Lasso estimator in
\citet{Jank17}.

For the theoretical analysis, we
define $T=\max_{j\in [p]} T_j$
as the maximum tuning parameter, 
 $\mb{R}=\corr(\bd{x})$ as the correlation matrix of $\bd{x}$, 
$A_j^*=\supp(\bd{\alpha}_j^*)$ as
the support set of $\bd{\alpha}_j^*$ given in~\eqref{lm, projection},
$m_j=\min_{i\in A_j^*} |\mb{\Omega}_{ij}|$
as the minimum absolute value of the off-diagonal nonzero
entries of $\mb{\Omega}$ in its $j$-th column,
and 
$s_j=
\|\mb{\Omega}_{*j} \|_0$ as the sparsity of the $j$-th column of $\mb{\Omega}$.
We define the following class of precision matrices:
\[
\mathcal{G}(s)=\left\{\mb{\Omega}\in\mathbb{R}^{p\times p}: \max_{j\in [p]}\| \mb{\Omega}_{*j}\|_0\le s, \text{~\ref{Theta bound} given below is satisfied}\right\}.
\]

We introduce some assumptions frequently used in our theoretical analysis.

\iffalse
To regularize $\bd{\alpha}$, it is better to scale $\mb{X}\in \mathbb{R}^{n\times p}$ to have $\sqrt{n}$-normalized columns. Specifically, we replace \eqref{regularize alpha} by
\[
\mb{Z}=\mb{X}\widehat{\mb{\Gamma}}^{-1/2},
\] 
where $\widehat{\mb{\Gamma}}=\diag(\widehat{\mb{\Sigma}})$ and $\widehat{\mb{\Sigma}}=\mb{X}^\top\mb{X}/n$.
Then,
\[
\mb{X}_{*j}-\mb{X}_{*\setminus j} \bd{\alpha} 
=\mb{Z}_{*j}\widehat{\mb{\Gamma}}_{jj}^{1/2}
-\mb{Z}_{*\setminus j} \widehat{\mb{\Gamma}}_{\setminus j,\setminus j}^{1/2} \bd{\alpha} 
\]

$\bd{\beta}=\widehat{\mb{\Gamma}}_{\setminus j,\setminus j}^{1/2} \bd{\alpha} \widehat{\mb{\Gamma}}_{jj}^{-1/2}$

\fi

\begin{enumerate}[label={(A\arabic*)}]
\item\label{Theta bound} (Bounded spectrum) $\kappa_1^{-1}\le \lambda_{\min}(\mb{\Omega})\le\lambda_{\max}(\mb{\Omega}) \le \kappa_2$ with positive constants $\kappa_1,\kappa_2$.

\item\label{subGauss} (Sub-Gaussianity) There exists a constant $K>0$ such that
\[
\sup_{\bd{v}\in \mathbb{R}^p:\|\bd{v}\|_2\le 1}E\left[\exp(|\bd{v}^\top\bd{x}|^2/K^2)\right]\le 2. 
\]

\item\label{assmp: bounded Tj}  
$(s_j-1)\vee 1\le T_j\le p-1$ for all $j\in [p]$.

\item\label{(C1)} $\gamma_T^{}<1-c_1$ with a constant $c_1>0$, 
$k_j\ge  \log_{1-\frac{c_1}{2}} \sqrt{T_j(\log p)/n}$
for all $j\in[p]$, and $T(\log p)/n=o(1)$,
where $
\gamma_T^{}
:=[2\theta_{T,T}+(1+\sqrt{2})\theta_{T,T}^2]c_\kappa^2+
(1+\sqrt{2})\theta_{T,T}c_\kappa
$ with
$\theta_{T,T}:=\max_{\substack{A,B\in S_{T}: A\cap B=\emptyset}} 
\| \mb{R}_{AB}\|_2 $,
 $S_{T}:=\{A\subseteq [p]:|A|\le T\}$, and $c_\kappa:=\kappa_1\kappa_2$.

\item\label{(C2)} $T\mu<1/4-c_2$ with a constant $c_2>0$, 
$\min_{j\in [p]}k_j\ge \log_{\frac{3-6c_2}{3+2c_2}}\sqrt{(\log p)/n}$,
and $T\sqrt{(\log p)/n}=o(1)$, where $\mu := \max_{i\ne j}|\mb{R}_{ij}|$.
\end{enumerate}

Assumptions~\ref{Theta bound} and \ref{subGauss} are commonly used in the literature on large precision matrix estimation, such as in \citet{Jank17} for Nodewise Lasso. Assumption~\ref{assmp: bounded Tj} requires that the tuning parameter \(T_j\) is not smaller than the true support size \(|A_j^*| = s_j - 1\), which is satisfied with high probability if \(T_j\) is selected optimally using the HBIC~\eqref{HBIC: Tj*} (see Lemma~\ref{Lemma2: A_hat=A}). Assumptions~\ref{(C1)} and \ref{(C2)} translate the corresponding assumptions in Theorems 4(i) and 12(i) of \citet{Huang18}, used to prove the consistency of their fixed design \(L_0\)-penalized linear regression estimation, to the random design setting in our context.

It is important to highlight that although our precision matrix estimation is related to linear regression, the theoretical analysis requires a more involved treatment due to its nonlinear nature and the fact that the SDAR theory developed by \citet{Huang18} for fixed design linear regression is not directly applicable to the random design used in nodewise regression. Unless specified otherwise, \(\widehat{\mathbf{\Omega}}\) in this section refers to the Nodewise Loreg estimator obtained from Algorithm~\ref{alg: Nodewise Loreg}.

\subsection{Convergence rates and support recovery}\label{subsec:  Convergence rates and support recovery}

We present the following theorem for the convergence rates of Nodewise Loreg estimator $\widehat{\mb{\Omega}}$
in various matrix norms.
\begin{theorem}[Convergence rates]\label{thm: consistency}
Suppose that \ref{Theta bound}-\ref{assmp: bounded Tj} hold.
For any  constant $C>0$, there exists a constant $M>0$ such that,
if \ref{(C1)} or \ref{(C2)} holds, then
\[
\inf_{\mb{\Omega}\in\mathcal{G}(T+1)} P\left(\| \widehat{\mb{\Omega}}-\mb{\Omega}\|_1\le M T\sqrt{(\log p)/n}\right)=1-O(p^{-C})
\]
and
$$
\inf_{\mb{\Omega}\in\mathcal{G}(T+1)} P\left( \max_{j\in[p]}\| \widehat{\mb{\Omega}}_{*j}-\mb{\Omega}_{*j}\|_2\le M \sqrt{T(\log p)/n}\right)	=1-O(p^{-C}),
$$
and
if \ref{(C2)} holds, then
$$
\inf_{\mb{\Omega}\in\mathcal{G}(T+1)} P\left(\| \widehat{\mb{\Omega}}-\mb{\Omega}\|_{\max}\le M\sqrt{(\log p)/n}\right)=1-O(p^{-C}).
$$
\end{theorem}	

\begin{remark}\label{remark: rate}
From Theorem~\ref{thm: consistency},
our Nodewise Loreg estimator attains the same order of convergence rates
as Nodewise Lasso estimator in the $L_1$ norm and the max column-wise $L_2$ norm
\citep[][Theorem~2.4]{VGeer14}.
Convergence rates in the max norm and the scaled Frobenius norm
can be obtained by
$
\| \widehat{\mb{\Omega}}-\mb{\Omega}\|_{\max}\vee p^{-1/2}\| \widehat{\mb{\Omega}}-\mb{\Omega} \|_F \le\max_{j\in[p]}\|\widehat{\mb{\Omega}}_{*j}-\mb{\Omega}_{*j} \|_2
=O_P(\sqrt{T(\log p)/n}).
$
However, our Nodewise Loreg estimator can achieve a faster convergence rate in the max norm, $\| \widehat{\mb{\Omega}}-\mb{\Omega}\|_{\max}=O_P(\sqrt{\log p/n})$, under assumption~\ref{(C2)}.
In contrast, this order of the max-norm  rate is only shown to be attainable by the desparsified (debiased) Nodewise Lasso estimator, which is not a sparse matrix and needs a post-thresholding that additionally requires $T(\log p)/\sqrt{n}=o(1)$ \citep[][Sections 3.2 and 3.4]{Jank17}.
The rates in the $L_1$ norm
and the scaled Frobenius norm
and the faster  max-norm rate
of our Nodewise Loreg estimator
are faster than
the minimax lower bounds, 
which are of the order
$\|\mb{\Omega} \|_1 T\sqrt{(\log p)/n}$, 
$\|\mb{\Omega} \|_1 \sqrt{T(\log p)/n}$,
and 
$\sqrt{(\log p)/n}\vee T (\log p)/n$, respectively,
as shown in 
\citet[][Theorems 4.1 and 6.1]{CLZ16}
and
\citet[][Theorem~5]{Zhao15}.
Our faster  rates
are  due to the 
additional condition
\ref{(C1)} or \ref{(C2)}.
These two conditions
 are translated versions of those in Theorems~4(i) and 12(i) of \citet{Huang18},
which are used to establish the consistency of
their fixed design $L_0$-penalized linear regression estimation, 
to the random design setting in our context.
\end{remark}

Next, we study
the support recovery of $\mb{\Omega}$
by $\widehat{\mb{\Omega}}$.
Before that,
we introduce  two lemmas regarding the support recovery of $\{A_j^*\}_{j\in[p]}$.

\begin{lem}\label{Lemma: A_hat=A}
Suppose that \ref{Theta bound}-\ref{assmp: bounded Tj} hold.
For any constant $C > 0$,
assume that 
\ref{(C1)} holds with $m_j\ge M_1\sqrt{T_j(\log p)/n} $
for all $j\in[p]$, or \ref{(C2)} holds with
$\min_{j\in[p]}m_j\ge M_2\sqrt{(\log p)/n}$,
where $M_1$ and $M_2$ are sufficiently
large positive constants dependent on $C$.
Then, 
$$
\inf_{\mb{\Omega}\in\mathcal{G}(T+1)}P(A_j^*\subseteq\widehat{A}_j, \forall j\in [p])=1-O(p^{-C}).
$$
If $T_j=s_j-1$ for all $j\in [p]$,
then 
$$
\inf_{\mb{\Omega}\in\mathcal{G}(T+1)}P(\widehat{A}_j=A_j^*,\forall j\in [p])=1-O(p^{-C}).
$$
\end{lem}	

Lemma~\ref{Lemma2: A_hat=A} given below shows that
the event 
$\{T_j=s_j-1,\forall j\in[p]\}$ is achievable with high probability,
when $T_j=T_j^*$, which is selected from the HBIC~\eqref{HBIC: Tj*}.
Denote $T_{\max}=\max_{j\in[p]}T_{\max,j}$, and
let $\gamma_{T_{\max}}^{}$ be the $\gamma_T^{}$ defined in \ref{(C1)}
with $T$ replaced by $T_{\max}$.

\begin{lem}\label{Lemma2: A_hat=A}
Suppose that \ref{Theta bound} and \ref{subGauss} hold, 
$\gamma_{T_{\max}}^{}<1-c$ with a constant $c>0$, 
 $T_{\max}(\log p)/n=o(1)$, 
 $s(\log p)\log(\log n)/n=o(1)$,
 and
{\color{black}$k_j(T_j)\ge  \log_{1-\frac{c}{2}} \sqrt{T_j(\log p)/n}$}
and $0\le T_j\le T_{\max,j}\in [(s_j-1)\vee 1,p-1]$
 for all $j\in [p]$.
For any constant $C>0$, if $m_j\ge M\sqrt{(T_{\max,j}\vee\log\log n)(\log p)/n}$ for all $j\in[p]$
with a sufficiently large constant $M>0$ dependent $C$,
then we have 
\be\label{eq: T*=s-1}
\inf_{\mb{\Omega}\in\mathcal{G}(s)} P(T_j^*=s_j-1, \forall j\in [p])=1-O(p^{-C})
\ee
and
$
\inf_{\mb{\Omega}\in\mathcal{G}(s)}P\left(\widehat{A}_j(k_j(T_j=T_j^*))=A_j^*,\forall j\in [p]\right)=1-O(p^{-C}).
$
\end{lem}

Now we present the result of Nodewide Loreg estimator $\widehat{\mb{\Omega}}$ on the support recovery of~$\mb{\Omega}$.

\begin{theorem}[Support  recovery]\label{thm: Sparsistency}
Suppose that \ref{Theta bound}-\ref{assmp: bounded Tj} hold.
For any constant $C>0$,
assume that \ref{(C1)} holds with $m_j\ge M_1\sqrt{T_j(\log p)/n} $ 
for all $j\in[p]$,
or \ref{(C2)} holds with
$\min_{j\in[p]}m_j\ge M_2\sqrt{(\log p)/n}$,
where $M_1$ and $M_2$ are sufficiently
large positive constants dependent on $C$.
Then, we have 
$$
\inf_{\mb{\Omega}\in\mathcal{G}(T+1)} P\left(\supp(\mb{\Omega})\subseteq \supp(\widehat{\mb{\Omega}})\right)=1-O(p^{-C})
$$
and
$
\inf_{\mb{\Omega}\in\mathcal{G}(T+1)}P\left(\sign(\widehat{\mb{\Omega}}_{ij})=\sign(\mb{\Omega}_{ij}),
\forall (i,j)\in \supp(\mb{\Omega})\right)=1-O(p^{-C}).
$
If $T_j=s_j-1$ for all $j\in [p]$, then we have 
\be\label{eq: supp_hat=supp}
\inf_{\mb{\Omega}\in\mathcal{G}(T+1)} P\left(\supp(\widehat{\mb{\Omega}})=\supp(\mb{\Omega})\right)=1-O(p^{-C}).
\ee
\end{theorem}

For the sparsistency of $\widehat{\mb{\Omega}}$
given in \eqref{eq: supp_hat=supp} in Theorem~\ref{thm: Sparsistency},
the condition 
$\{T_j=s_j-1,\forall j\in[p]\}$ is
achievable with 
high probability if $T_j$ is set to $T_j^*$
selected from the HBIC~\eqref{HBIC: Tj*},
due to \eqref{eq: T*=s-1} in Lemma~\ref{Lemma2: A_hat=A}.
Thus, we obtain
the following corollary. 

\begin{cor}[Sparsistency]\label{cor: Sparsistency}
Under the conditions of Lemma~\ref{Lemma2: A_hat=A},
when $T_j=T_j^*$ for all $j\in [p]$,
we have 
$\inf_{\mb{\Omega}\in\mathcal{G}(s)} P\big(\supp(\widehat{\mb{\Omega}})=\supp(\mb{\Omega})\big)=1-O(p^{-C})
$.
\end{cor}

\subsection{Asymptotic normality}
We consider the asymptotic normality of Nodewise Loreg estimator $\widehat{\mb{\Omega}}$.
We write
\begin{align}
\sqrt{n}(\widehat{\mb{\Omega}}_{ij}-\mb{\Omega}_{ij})
&=-\sqrt{n}\big[\mb{\Omega}_{ii}\mb{\Omega}_{jj}(\bd{\epsilon}_{i\parallel \bd{x}_{\widehat{A}_j}}^\top+\bd{\epsilon}_{i\parallel\epsilon_j}^\top )\bd{\epsilon}_j/n
 - \mb{\Omega}_{ij}\big]+r_{ij}~~\text{for}~~i\in [p], \label{diff in Omega for AN}\\
& =-\sqrt{n}
[(\mb{\Sigma}_{\widehat{A}_j^+\widehat{A}_j^+}^{-1})_{\hat{i}*}
\widehat{\mb{\Sigma}}_{\widehat{A}_j^+\widehat{A}_j^+}
(\mb{\Sigma}_{\widehat{A}_j^+\widehat{A}_j^+}^{-1})_{*\hat{j}}-\mb{\Omega}_{ij}]+\tilde{r}_{ij}~~\text{for}~~i\in \widehat{A}_j^+,
\label{diff in Omega for AN, simplified}
\end{align}
where 
$\bd{\epsilon}_{i\parallel \bd{x}_{\widehat{A}_j}}
=[\mb{\Omega}_{ii}^{-1}(\bd{e}_i)_{\widehat{A}_j}^\top
\mb{\Sigma}_{\widehat{A}_j\widehat{A}_j}^{-1}\mb{X}_{*\widehat{A}_j}^\top]^\top\in \mathbb{R}^n$,
$\bd{e}_i\in \mathbb{R}^p$ with $1$ on the $i$-th entry and 0
elsewhere,
$\bd{\epsilon}_{i\parallel\epsilon_j}=\mb{\Omega}_{ii}^{-1}\mb{\Omega}_{ij}\bd{\epsilon}_j$,
$\bd{\epsilon}_j
=\mb{X}_{*j}-\mb{X}_{*\setminus j}\bd{\alpha}_j^*
$,
and $\hat{i}$ and $\hat{j}$ are the positions of $i$ and $j$
in $\widehat{A}_j^+=\widehat{A}_j\cup\{j\}$
when its elements are sorted in ascending order.
The remainder terms
$r_{ij}$ and $\tilde{r}_{ij}$ satisfy
$P(r_{ij}=\tilde{r}_{ij}, \forall i\in\widehat{A}_j^+,j\in [p])\to 1$
and
$
\max_{1\le i,j\le p} |r_{ij}|= O_P(T(\log p)/\sqrt{n})=o_P(1)
$
under certain conditions (see Lemma~\ref{thm: normality}), and thus are negligible.

Consider the main term in \eqref{diff in Omega for AN}.
Define 
$\epsilon_{i \parallel \bd{x}_{\widehat{A}_j}}=
\mb{\Omega}_{ii}^{-1}(\bd{e}_i)_{\widehat{A}_j}^\top
\mb{\Sigma}_{\widehat{A}_j\widehat{A}_j}^{-1}\bd{x}_{\widehat{A}_j}$,
$
\epsilon_{i\parallel\epsilon_j}
=\mb{\Omega}_{ii}^{-1}\mb{\Omega}_{ij}\epsilon_j
$,
and
$\epsilon_j=x_j-\bd{x}_{\setminus j }^\top \bd{\alpha}_j^*$.
We let $\bd{x}$ be independent of $\widehat{A}_j$,
despite that $\widehat{A}_j$ may change with
values of $\mb{X}\in\mathbb{R}^{n\times p}$,
whose rows, however, are i.i.d. copies of $\bd{x}$.
The scaled version of of the first component,
$\bd{\epsilon}_{i\parallel \bd{x}_{\widehat{A}_j}}^\top\bd{\epsilon}_j/n$, is the sample mean of $\epsilon_{i \parallel \bd{x}_{\widehat{A}_j}} \epsilon_j$ with
$E[\epsilon_{i \parallel \bd{x}_{\widehat{A}_j}} \epsilon_j|\widehat{A}_j]=0$.
Rewrite
$\epsilon_{i \parallel \bd{x}_{\widehat{A}_j}}
=\mb{\Omega}_{ii}^{-1}\mb{\Omega}_{*i}^{\top}\mb{\Sigma}_{*\widehat{A}_j}\mb{\Sigma}_{\widehat{A}_j\widehat{A}_j}^{-1}\bd{x}_{\widehat{A}_j}
=\mb{\Omega}_{ii}^{-1}E[\mb{\Omega}_{*i}^{\top}\bd{x} \bd{b}_{\widehat{A}_j}^\top|\widehat{A}_j]
\bd{b}_{\widehat{A}_j}
=E[\epsilon_i \bd{b}_{\widehat{A}_j}^\top|\widehat{A}_j]\bd{b}_{\widehat{A}_j}$
with $\bd{b}_{\widehat{A}_j}=\mb{\Sigma}_{\widehat{A}_j\widehat{A}_j}^{-1/2}\bd{x}_{\widehat{A}_j}$.
Given $\widehat{A}_j$,
it implies that
$\epsilon_{i \parallel \bd{x}_{\widehat{A}_j}}$
 is the orthogonal projection of $\epsilon_i$ onto $\lspan(\bd{x}_{\widehat{A}_j}^\top)$, which is the subspace spanned by 
entries of $\bd{x}_{\widehat{A}_j}$ in
$(\mathcal{L}^2(\mathbb{R}), E)$. Here, $(\mathcal{L}^2(\mathbb{R}), E)$ is
 the $\mathcal{L}^2$ space of real random variables with expectation as the inner product.
 Also note that
$\epsilon_j$ is the orthogonal rejection of
$x_j$ from $\lspan(\bd{x}_{\setminus j}^\top)$.
Thus, $\epsilon_{i \parallel \bd{x}_{\widehat{A}_j}}  \perp \epsilon_j$, i.e., 
$E(\epsilon_{i \parallel \bd{x}_{\widehat{A}_j}} \epsilon_{\widehat{A}_j}|\widehat{A}_j)=0$.
The scaled version of the second component,
$\bd{\epsilon}_{i\parallel\epsilon_j}^\top \bd{\epsilon}_j/n$,
is the sample mean of $\epsilon_{i\parallel\epsilon_j}\epsilon_j$
with $E[\epsilon_{i\parallel\epsilon_j}\epsilon_j]=
E[\epsilon_i\epsilon_j]=\mb{\Omega}_{ii}^{-1}\mb{\Omega}_{jj}^{-1}\mb{\Omega}_{ij}$, which equals the  scaled negative version of the third component,
since $\epsilon_{i\parallel\epsilon_j}=E[\epsilon_i\epsilon_j/\sd(\epsilon_j)]\epsilon_j/\sd(\epsilon_j)$ is the orthogonal projection of
$\epsilon_i$ onto $\epsilon_j$ in space $(\mathcal{L}^2(\mathbb{R}), E)$.
Hence,  
the asymptotic normality of $\widehat{\mb{\Omega}}_{ij}$
can be derived using 
the central limit theorem, assuming
$\widehat{A}_j$ is a fixed set given $n$ and $p$.
However, 
the i.i.d. assumption may not hold for
entries of $\bd{\epsilon}_{i\parallel \bd{x}_{\widehat{A}_j}}$
in the first scaled component 
$\bd{\epsilon}_{i\parallel \bd{x}_{\widehat{A}_j}}^\top\bd{\epsilon}_j/n$, as 
$\widehat{A}_j$ is a random set dependent on data $\mb{X}$.

We consider the asymptotic normality under the assumption
that  the event $\mathcal{E}_j=\{A_j^*\subseteq \widehat{A}_j =\widetilde{A}_j(n,p)\}$
holds with high probability,
where $\widetilde{A}_j(n,p)$ is a fixed set given $n$ and $p$.
Lemmas~\ref{Lemma: A_hat=A} and \ref{Lemma2: A_hat=A} show that
event $\mathcal{E}_j$ is achievable with high probability.
One might relax the condition $\widehat{A}_j\,{=}\,\widetilde{A}_j(n,p)$ by
 treating 
the entries of $\bd{\epsilon}_{i\parallel \bd{x}_{\widehat{A}_j}}$
as exchangeable variables,
and then apply the central limit theorem of
exchangeable variables~\citep{jiang2002empirical}.
This theoretical derivation is more involved and thus is not pursued in this paper. 
For simplicity, $\widetilde{A}_j(n,p)$  is abbreviated as $\widetilde{A}_j$ by omitting its  dependence on $n$ and $p$.
We define
$\sigma_{i,\widetilde{A}_j}=\sd(\xi_{i,\widetilde{A}_j}\epsilon_j)$
and 
$\sigma_{i,\widehat{A}_j}=\sd(\xi_{i,\widehat{A}_j}\epsilon_j|\widehat{A}_j)$,
where $
\xi_{i,A}
=-\mb{\Omega}_{ii}\mb{\Omega}_{jj}(\epsilon_{i\parallel \bd{x}_{A}}+\epsilon_{i\parallel \epsilon_j})I(i\in A)
-\mb{\Omega}_{ii}\mb{\Omega}_{jj}\epsilon_iI(i\notin A)
$ 
and
$\epsilon_{i \parallel \bd{x}_A}
=
\mb{\Omega}_{ii}^{-1}(\bd{e}_i)_A^\top
\mb{\Sigma}_{AA}^{-1}\bd{x}_A$
with $A\in \{\widetilde{A}_j,\widehat{A}_j\}$.

\begin{lem}\label{thm: normality}
Suppose that \ref{Theta bound}-\ref{assmp: bounded Tj} hold,
and either \ref{(C1)} or \ref{(C2)} is satisfied.
For $\mb{\Omega}\in\mathcal{G}(T+1)$,
if $P(A_j^*\subseteq \widehat{A}_j,\forall j\in [p])\to 1$,
then 
$P(r_{ij}=\tilde{r}_{ij}, \forall i\in\widehat{A}_j^+,j\in [p])\to 1$ and 
$\max_{1\le i,j\le p}|r_{ij}|=O_P(T(\log p)/\sqrt{n})$.
Further, if
$P(\cap_{j\in [p]}\mathcal{E}_j){\to 1}$,
$T(\log p)/\sqrt{n}=o(1)$, and
$\min_{i\in \widehat{A}_j^+,j\in [p]}\sigma_{i,\widetilde{A}_j}\ge \omega$
with a constant $\omega>0$, 
then
\be\label{thm eqn: AN Omega all}
\sup_{ i\in \widehat{A}_j^+,j\in[p],z\in\mathbb{R}}\left|P({\sqrt{n}(\widehat{\mb{\Omega}}_{ij}-\mb{\Omega}_{ij})}/\sigma_{i,\widehat{A}_j}\le z)-\Phi(z)\right|\to 0.
\ee
\end{lem}

The assumption
 $T(\log p)/\sqrt{n}=o(1)$ and
 a condition similar to
 $\min_{i\in \widehat{A}_j^+,j\in [p]}\sigma_{i,\widetilde{A}_j}\ge \omega$
are also required in
the
asymptotic normality theorem of 
the desparsified Nodewise Lasso estimator in \citet{Jank17} (see their conditions A2* and A4).
From Lemmas~\ref{Lemma: A_hat=A} and~\ref{Lemma2: A_hat=A},
the condition $P(\cap_{j\in [p]}\mathcal{E}_j){\to 1}$
in Lemma~\ref{thm: normality}
 is achievable
under certain conditions for 
minimum magnitudes of nonzero entries $\{m_j\}_{j\in[p]}$
and
tuning parameters $\{T_j\}_{j\in[p]}$. Thus, we have the following result.

\begin{theorem}[Asymptotic normality of $\widehat{\mb{\Omega}}$]\label{Corollary: AN}
Suppose that \ref{Theta bound} and \ref{subGauss} hold,
$\gamma_{T_{\max}}^{}<1-c$ with a constant $c>0$, 
 $T_{\max}(\log p)/n=o(1)$, and
 {\color{black}$s(\log p)/\sqrt{n}=o(1)$}.
 For all $j\in [p]$, assume that
$k_j(T_j)\ge\log_{1-\frac{c}{2}} \sqrt{T_j(\log p)/n}$, 
$0\le T_j\le T_{\max,j}\in [(s_j-1)\vee 1,p-1]$,
and $m_j\ge M\sqrt{(T_{\max,j}\vee\log\log n)(\log p)/n}$
with a sufficiently large constant $M>0$. 
Let {\color{black}$\widehat{A}_j=\widehat{A}_j(k_j(T_j=T_j^*))$ for all $j\in [p]$} with $T_j^*$ given in \eqref{HBIC: Tj*}.
Assume that
$\min_{i\in \widehat{A}_j^+,j\in [p]}\sigma_{i,A_j^*}\ge \omega$
with a constant $\omega>0$.
Then, we have 
\be\label{thm eqn: AN Omega all}
\sup_{\mb{\Omega}\in \mathcal{G}(s), i\in \widehat{A}_j^+,j\in[p],z\in\mathbb{R}}\left|P({\sqrt{n}(\widehat{\mb{\Omega}}_{ij}-\mb{\Omega}_{ij})}/\sigma_{i,\widehat{A}_j}\le z)-\Phi(z)\right|\to 0.
\ee
\end{theorem}

Lemma~\ref{thm: normality} and Theorem~\ref{Corollary: AN} demonstrate that the Nodewise Loreg estimator $\widehat{\mb{\Omega}}$
achieves asymptotic normality without the need for debiasing that is required by the Nodewise Lasso estimator \citep{Jank17}. 
Given the above asymptotic normality and the sparsistency  in Section~\ref{subsec:  Convergence rates and support recovery}, we can conclude that 
the Nodewise Loreg estimator $\widehat{\mb{\Omega}}$ is asymptotically unbiased.

We now consider the estimation of the asymptotic variance $\sigma_{i,\widehat{A}_j}^2$.
From \eqref{diff in Omega for AN, simplified} and Lemma~\ref{thm: normality},
it holds with high probability that
$
\sigma_{i,\widehat{A}_j}^2=\var((\mb{\Sigma}_{\widehat{A}_j^+\widehat{A}_j^+}^{-1})_{\hat{i}*}\bd{x}_{\widehat{A}_j^+}\bd{x}_{\widehat{A}_j^+}^\top(\mb{\Sigma}_{\widehat{A}_j^+\widehat{A}_j^+}^{-1})_{*\hat{j}}
|\widehat{A}_j)
$ for $i\in \widehat{A}_j^+$.
Thus, we can estimate $\sigma_{i,\widehat{A}_j}^2$ by
\be\label{sigma_i,A_j, general}
\widehat{\sigma}_{i,\widehat{A}_j}^2:=\frac{1}{n}\sum_{k=1}^n[(\widehat{\mb{\Sigma}}_{\widehat{A}_j^+\widehat{A}_j^+}^{-1})_{\hat{i}*}\mb{X}_{k,\widehat{A}_j^+}^\top\mb{X}_{k,\widehat{A}_j^+}(\widehat{\mb{\Sigma}}_{\widehat{A}_j^+\widehat{A}_j^+}^{-1})_{*\hat{j}}]^2
-\frac{1}{2}(\widehat{\mb{\Omega}}_{ij}^2+\widehat{\mb{\Omega}}_{ji}^2)
~~\text{for}~~ i\in \widehat{A}_j^+.
\ee
Further, if $\bd{x}$ follows a $p$-variate Gaussian distribution,
then with high probability
$\sigma_{i,\widehat{A}_j}^2$ has the closed form
$\sigma_{i,\widehat{A}_j}^2=(\mb{\Sigma}_{\widehat{A}_j^+\widehat{A}_j^+}^{-1})_{\hat{i}\hat{i}}(\mb{\Sigma}_{\widehat{A}_j^+\widehat{A}_j^+}^{-1})_{\hat{j}\hat{j}}+(\mb{\Sigma}_{\widehat{A}_j^+\widehat{A}_j^+}^{-1})_{\hat{i}\hat{j}}^2$ for $i\in \widehat{A}_j^+$.
Then, we estimate $\sigma_{i,\widehat{A}_j}^2$ by
\be\label{sigma_i,A_j, Gaussian}
\widehat{\sigma}_{i,\widehat{A}_j}^2=(\widehat{\mb{\Sigma}}_{\widehat{A}_j^+\widehat{A}_j^+}^{-1})_{\hat{i}\hat{i}}(\widehat{\mb{\Sigma}}_{\widehat{A}_j^+\widehat{A}_j^+}^{-1})_{\hat{j}\hat{j}}+(\widehat{\mb{\Sigma}}_{\widehat{A}_j^+\widehat{A}_j^+}^{-1})_{\hat{i}\hat{j}}^2
~~\text{for}~~ i\in \widehat{A}_j^+.
\ee
The following theorem shows that 
$\widehat{\sigma}_{i,\widehat{A}_j}^2$ is a consistent estimator of $\sigma_{i,\widehat{A}_j}^2$.

\begin{theorem}[Consistency of $\widehat{\sigma}_{i,\widehat{A}_j}^2$]
\label{thm: consistency of sigma_iAj}
Suppose that \ref{Theta bound}-\ref{assmp: bounded Tj} hold,
either \ref{(C1)} or \ref{(C2)} is satisfied,
and $P(\cap_{j\in [p]}\mathcal{E}_j)\to 1$.
We have the following results.
\begin{enumerate}[label=(\roman*),leftmargin=*, itemsep=0pt, topsep=0pt, partopsep=0pt, parsep=0pt]
\item\label{thm (i): consistency of sigma_iAj} (Sub-Gaussian cases)
Assume that
$T(\log p)/\sqrt{n}=o(1)$, and 
$[\log(p\vee n)]^4/n^{1-c}=o(1)$ with a constant $c>0$.
Let $\widehat{\sigma}_{i,\widehat{A}_j}^2$ be the estimator defined in  \eqref{sigma_i,A_j, general}.
Then, for all $\varepsilon>0$,
\[
P\Big(\max_{i\in \widehat{A}_j^+,j\in[p]} |\widehat{\sigma}_{i,\widehat{A}_j}^2-\sigma_{i,\widehat{A}_j}^2|\ge \varepsilon\Big)
\to 0.
\]
\item\label{thm (ii): consistency of sigma_iAj} (Gaussian cases)
Assume that $\bd{x}$ follows a $p$-variate Gaussian distribution.
Let $\widehat{\sigma}_{i,\widehat{A}_j}^2$ be the estimator defined in  \eqref{sigma_i,A_j, Gaussian}.
Then, we have 
\[
\max_{i\in \widehat{A}_j^+,j\in[p]} |\widehat{\sigma}_{i,\widehat{A}_j}^2-\sigma_{i,\widehat{A}_j}^2|
=O_P(\sqrt{T(\log p)/n}).
\]
\end{enumerate}
\end{theorem}

\subsection{Discussion on a desparsified Nodewise Loreg estimator}
We can  construct 
a desparsified version of the Nodewise Loreg
estimator 
in the same form of the desparsified (debiased) Nodewise Lasso estimator
given in \citet{Jank17}. Specifically,
we have 
\be
\widehat{\mb{\Omega}}_{ij}-\widehat{\mb{\Omega}}_{*i}^\top(\widehat{\mb{\Sigma}}\widehat{\mb{\Omega}}_{*j}-\bd{e}_j)  -\mb{\Omega}_{ij}
=-\mb{\Omega}_{*i}^\top(\widehat{\mb{\Sigma}}-\mb{\Sigma})
\mb{\Omega}_{*j}+\Delta_{ij}/\sqrt{n}
.
\label{Theta-Theta2d-Theta}
\ee
where $\Delta_{ij}=o_P(1)$ under certain conditions (see Theorem~\ref{thm: normality 2}).
Thus, the desparsified Nodewise Loreg estimator is  defined as
 \be
 \widehat{\mb{T}}=\widehat{\mb{\Omega}}-\widehat{\mb{\Omega}}^\top(\widehat{\mb{\Sigma}}\widehat{\mb{\Omega}}-\mb{I})
 =\widehat{\mb{\Omega}}+\widehat{\mb{\Omega}}^\top-
 \widehat{\mb{\Omega}}^\top  \widehat{\mb{\Sigma}}   \widehat{\mb{\Omega}},
 \ee
 which employs  the same formula as the desparsified Nodewise Lasso estimator, 
 except for plugging in their respective undesparsified estimators as $\widehat{\mb{\Omega}}$.
The following theorem
shows that
the desparsified Nodewise
Loreg estimator
$\widehat{\mb{T}}$
enjoys asymptotic normality
without requiring
the condition $P( \widehat{A}_j\,{=}\,\widetilde{A}_j(n,p), \forall j \in [p])\to 1$
assumed for 
its  undesparsified counterpart $\widehat{\mb{\Omega}}$.

\begin{theorem}[Asymptotic normality of $\widehat{\mb{T}}$]\label{thm: normality 2}
Suppose that \ref{Theta bound}-\ref{assmp: bounded Tj} hold. 
Assume that 
\ref{(C1)} holds with $m_j\ge M_1\sqrt{T_j(\log p)/n} $
for all $j\in[p]$, or \ref{(C2)} holds with
$\min_{j\in[p]}m_j\ge M_2\sqrt{(\log p)/n}$,
where $M_1$ and $M_2$ are sufficiently
large positive constants.
Then uniformly for all $\mb{\Omega}\in\mathcal{G}(T+1)$,
we have $\max_{1\le i,j\le p}|\Delta_{ij}|=O_P(T(\log p)/\sqrt{n})$.
Further, if $T(\log p)/\sqrt{n}=o(1)$  and $\sigma_{ij}:=\sd(\mb{\Omega}_{*i}^\top \bd{x}\bd{x}^\top \mb{\Omega}_{*j})
%=\sd(\mb{\Omega}_{ii}\mb{\Omega}_{jj}\epsilon_i\epsilon_j)
\ge w$ with a constant $w>0$,
then
\[
\sup_{\mb{\Omega}\in \mathcal{G}(T+1),
1\le i,j\le p,
 z\in\mathbb{R}}|P(\sqrt{n}(\widehat{\mb{T}}_{ij}-\mb{\Omega}_{ij})/\sigma_{ij}\le z)-\Phi(z)|\to 0.
\]

\end{theorem}

For the asymptotic variance $\sigma_{ij}^2=\var(\mb{\Omega}_{*i}^\top \bd{x}\bd{x}^\top \mb{\Omega}_{*j})
$   in Theorem~\ref{thm: normality 2},
we  estimate it by 
\be\label{sigma_ij, general}
\widehat{\sigma}_{ij}^2:=\frac{1}{n}\sum_{k=1}^n(\widehat{\mb{\Omega}}_{*i}^\top\mb{X}_{k*}^\top\mb{X}_{k*}\widehat{\mb{\Omega}}_{*j})^2
-\frac{1}{2}(\widehat{\mb{\Omega}}_{ij}^2+\widehat{\mb{\Omega}}_{ji}^2).
\ee
If $\bd{x}$ follows a $p$-variate Gaussian distribution,
then we have
$\sigma_{ij}^2=\mb{\Omega}_{ii}\mb{\Omega}_{jj}+\mb{\Omega}_{ij}^2$ and estimate it by
\be\label{sigma_ij, Gaussian}
\widehat{\sigma}_{ij}^2=\widehat{\mb{\Omega}}_{ii}\widehat{\mb{\Omega}}_{jj}+\frac{1}{2}(\widehat{\mb{\Omega}}_{ij}^2+\widehat{\mb{\Omega}}_{ji}^2).
\ee
We establish the following consistency result for $\widehat{\sigma}_{ij}^2$ with respect  to $\sigma_{ij}^2$.

\begin{theorem}[Consistency of $\widehat{\sigma}_{ij}^2$]
\label{thm: consistency of sigma_ij}
Suppose that \ref{Theta bound}-\ref{assmp: bounded Tj} hold, and either \ref{(C1)} or \ref{(C2)}  is satisfied. We have the following results.
\begin{enumerate}[label=(\roman*),leftmargin=*, itemsep=0pt, topsep=0pt, partopsep=0pt, parsep=0pt]
\item\label{thm (i): consistency of sigma_ij} (Sub-Gaussian cases)
Assume that
$T(\log p)/\sqrt{n}=o(1)$, and 
$[\log(p\vee n)]^4/n^{1-c}=o(1)$ with a constant $c>0$.
Let $\widehat{\sigma}_{ij}^2$ be the estimator defined in  \eqref{sigma_ij, general}.
Then, for all $\varepsilon>0$,
\[
\sup_{\mb{\Omega}\in\mathcal{G}(T+1)} P\Big(\max_{1\le i,j\le p} |\widehat{\sigma}_{ij}^2-\sigma_{ij}^2|\ge \varepsilon\Big)
\to 0.
\]

\item\label{thm (ii): consistency of sigma_ij} (Gaussian cases)
Assume that $\bd{x}$ follows a $p$-variate Gaussian distribution.
Let $\widehat{\sigma}_{i,\widehat{A}_j}^2$ be the estimator defined in  \eqref{sigma_ij, Gaussian}.
Then, uniformly for all $\mb{\Omega}\in\mathcal{G}(T+1)$, we have 
\[
\max_{1\le i,j\le p} |\widehat{\sigma}_{ij}^2-\sigma_{ij}^2|=O_P(\sqrt{T(\log p)/n}).
\]
\end{enumerate}
\end{theorem}

We observe that the  asymptotic variance of desparsified Nodewise Loreg estimator $\widehat{\mb{T}}_{ij}$,
 $\sigma_{ij}^2:=\var(\mb{\Omega}_{*i}^\top \bd{x}\bd{x}^\top \mb{\Omega}_{*j})=\var(\mb{\Omega}_{ii}\mb{\Omega}_{jj}  \epsilon_i\epsilon_j)$,
is similar to  that  of undesparsified Nodewise Loreg estimator $\widehat{\mb{\Omega}}_{ij}$, 
  $\sigma_{i,\widehat{A}_j}^2:=\var(\xi_{i,\widehat{A}_j}\epsilon_j|\widehat{A}_j)=\var(\mb{\Omega}_{ii}\mb{\Omega}_{jj}  (\epsilon_{i\parallel \bd{x}_{\widehat{A}_j}}+\epsilon_{i\parallel \epsilon_j})\epsilon_j    |\widehat{A}_j)$.
Thus, we aim to compare these two asymptotic variances.
Note that $\sigma_{ij}^2$ is also the asymptotic variance of the desparsified Nodewise Lasso estimator \citep{Jank17}.

\begin{theorem}[Comparison of asymptotic variances]
\label{sigma_ij compare}
Assume that 
$\bd{x}$ follows a $p$-variate Gaussian distribution.
Given any $\widehat{A}_j$ such that $A_j^*\subseteq \widehat{A}_j$,
we have
$\sigma_{i,\widehat{A}_j}^2\le \sigma_{ij}^2$
for $i\in \widehat{A}_j^+$, 
where equality holds
if and only if $(A_i^*\cup\{i\})\setminus \widehat{A}_j^+=\emptyset$.
\end{theorem}

\begin{remark}
From Theorem~\ref{sigma_ij compare}, when $\bd{x}$ is Gaussian,
$\sigma_{i,\widehat{A}_j}^2$, the asymptotic variance of undesparsified Nodewise Loreg estimator $\widehat{\mb{\Omega}}_{ij}$,
is dominated by
 $\sigma_{ij}^2$, the asymptotic variance
shared by both the desparsified Nodewise Loreg and desparsified Nodewise Lasso estimators $\widehat{\mb{T}}_{ij}$.
 This suggests a potential for more powerful statistical inference, 
 such as discovering nonzero $\mb{\Omega}_{ij}$ 
through  hypothesis testing,
when $\widehat{\mb{\Omega}}_{ij}$ and $\widehat{\mb{T}}_{ij}$ are very close.
It should be noted that
Theorem~\ref{sigma_ij compare} may not hold when $\bd{x}$ is non-Gaussian.
A counterexample is provided  in Supplementary Materials. 
Nonetheless,  in many practical applications, such as genomic data analysis, non-Gaussian data are often transformed to approximate Gaussian distributions \citep{box1964analysis,durbin2002variance,feng2016note}, 
potentially retaining the advantage when
 $\sigma_{i,\widehat{A}_j}^2$ is smaller.

\end{remark}

\section{Simulation Studies}\label{s: simulations}

\subsection{Simulation settings}
We numerically compare our Nodewise Loreg with Nodewise Lasso, CLIME,
and GLasso.
The following four graph models are considered for the precision matrix $\mb{\Omega}$.

\begin{itemize}[leftmargin=*, itemsep=0pt]
\item \textbf{Band graph}: $\mb{\Omega}_{ij}=I(|i-j|=0)+0.5I(|i-j|=1)+0.3I(|i-j|=2)$.
 % and a perturbation term that follows from uniform distribution $U(-0.05,0.05)$ is added to each non-zero off-diagonal element of $\bm \Omega$. 
 The resulting graph has $2p-3$ edges. 

 % Setting4: $\bm \Omega_{ij}=I(|i-j|=0)+0.6I(|i-j|=1)+0.6^2I(|i-j|=2)+0.6^3I(|i-j|=3)$.

\item \textbf{Random graph}: 
We add an edge between each pair of the $p$ nodes with probability $4/p$ independently.
The resulting graph has approximately $2(p-1)$ edges. 
Once the graph is obtained, 
we construct the adjacency matrix $\mb{A}$
and generate the precision matrix
\begin{align} \label{precision}
    \mb{\Omega} = \mb{A}+(|\lambda_{\min}(\mb{A})|+0.1)\mb{I}_{p\times p}.
\end{align}

\item  \textbf{Hub graph}: The $p$ nodes are evenly partitioned into $p/10$ disjoint groups. 
Within each group, one node is selected as the hub and we add an edge between the hub and every other node. The resulting graph has
$p - p/10$ edges. $\mb{\Omega}$ is generated using~\eqref{precision}.

\item  \textbf{Cluster graph}: 
The $p$ nodes are evenly partitioned into $p/10$ disjoint groups. 
Within each group, we add an edge between each pair of the 10 nodes with probability 0.6 independently.
The resulting graph has approximately
 $2.7p$ edges.  $\mb{\Omega}$ is generated by~\eqref{precision}.\end{itemize}

For each graph model, the precision matrix $\mb{\Omega}$ is fixed across all replications within the same simulation setting. Similar graph models have been considered in 
the literature \citep{TigerLiu17,pmlr-v51-wang16a,zhao2012huge}.
We generate the data matrix $\mb{X}\in\mathbb{R}^{n\times p}$,  consisting of $n$ rows that are i.i.d. copies of a $p$-dimensional random vector $\bd{x}$.
We consider the following Gaussian and sub-Gaussian cases for $\bd{x}$, also considered in \citet{Jank17}.

\begin{itemize}[leftmargin=*, itemsep=0pt]
\item {\bf Gaussian case}: $\bd{x}\sim \mathcal{N}(\bd{0},\mb{\Omega}^{-1})$.

\item {\bf Sub-Gaussian case}: Let $\bd{u}$ be a vector consisting of $p$ i.i.d. entries from a continuous uniform distribution on the interval $[-\sqrt{3},\sqrt{3}]$.
Generate $\bd{x}=\mb{\Omega}^{-1/2}\bd{u}$.
\end{itemize}

We mainly consider the four settings of sample size and variable dimension: $(n,p)\in \{(200,200),$ $(200,400),(400,200),(400,400)\}$. For each simulation setting, we conduct $100$ independent replications to calculate evaluation metrics. For figures assessing asymptotic normality, we increase the number of replications to $400$ and add cases of $n=800$ to ensure a clearer comparison of the histogram with the standard Gaussian curve.

We implement our Nodewise Loreg by invoking C++ subroutines through R for the SDAR algorithm. For the implementation of Nodewise Lasso, CLIME, and GLasso, we utilize the R packages \textit{glmnet}~\citep{glmnet2022}, \textit{flare}~\citep{flare2022}, and \textit{glasso}~\citep{GLasso2019}, which are coded in C++, C, and Fortran, respectively.
For our Nodewise Loreg, we select tuning parameters $\{T_j\}_{j=1}^p$ by minimizing  the HBIC~\eqref{HBIC general} as described in~\eqref{HBIC: Tj*}  with the maximum values $\{T_{\max,j}\}_{j=1}^p$ all set to  $20$.
For Nodewise Lasso,  tuning parameters $\{\lambda_j\}_{j=1}^p$ are also selected
by minimizing  the HBIC~\eqref{HBIC general} in which plugging their corresponding estimated active sets.
The candidate values for  $\{\lambda_j\}_{j=1}^p$ range from 0.02 to 2 with 
20 equal logarithmic spaced values.
For both CLIME and GLasso,
we select the optimal value of their respective tuning parameter $\lambda$ by minimizing the BIC criterion \citep{Yuan07}:
  \[
  \BIC(\lambda)=-\log |\widehat{\mb{\Omega}}(\lambda)|+\text{tr}(\widehat{\mb{\Omega}}(\lambda)\widehat{\mb{\Sigma}})+\frac{\log n}{n}\sum_{i\leq j}I([\widehat{\mb{\Omega}}(\lambda)]_{ij}),
  \]
 where $\widehat{\mb{\Omega}}(\lambda)$ is the estimated precision matrix using tuning parameter $\lambda$, and the candidate values of $\lambda$ is the same as those for Nodewise Lasso.

%    Due to the column-by-column estimating nature, we apply the minimum symmetrization $\widehat{\mb{\Omega}}_{ij}^{\text{S}}=\widehat{\mb{\Omega}}_{ij}^{\text{US}}I(|\widehat{\mb{\Omega}}_{ij}^{\text{US}}|\le |\widehat{\mb{\Omega}}_{ji}^{\text{US}}|)+\widehat{\mb{\Omega}}_{ji}^{\text{US}}I(|\widehat{\mb{\Omega}}_{ij}^{\text{US}}|> |\widehat{\mb{\Omega}}_{ji}^{\text{US}}|)$ to Nodewise Loreg and Nodewise Lasso as used in \citet{CaiClime} for CLIME, where $\widehat{\mb{\Omega}}^{\text{S}}$ and $\widehat{\mb{\Omega}}^{\text{US}}$ denote the symmetrized and unsymmetrized precision matrix estimators, respectively.
%{\color{black} GLasso symmetric issue.}
%Here, we use $\widehat{\mb{\Omega}}_{ij}$ and $\widehat{\mb{\Omega}}_{ij}^{(1)}$ to denote estimates of $\mb{\Omega}$ before and after symmetrization, obtained from Nodewise Loreg or Nodewise Lasso.

To utilize the asymptotic normality of  Nodewise Loreg and Nodewise Lasso,
we also consider to further threshold their estimates by multiple testing. 
Denote the null and alternative hypotheses at entry $(i,j)$ by $H_{ij,0}$: $\mb{\Omega}_{ij}=0$
and $H_{ij,1}$: $\mb{\Omega}_{ij}\ne 0$, respectively.
Let $S_L(\mb{M})$ denote the support set of  the off-diagonal lower triangular part of a square matrix~$\mb{M}$.
Define
$\mathcal{T}(\mb{M}_1|Z_0(\mb{M}_2),S_L(\mb{M}_3))$
as the thresholded version of a symmetric matrix $\mb{M}_1$ 
via multiple testing for $H_{ij,0}$ against $H_{ij,1}$ at all entries $(i,j)\in S_L(\mb{M}_3)$
based on the $Z$-score of $(\mb{M}_2)_{ij}$ under $H_{ij,0}$ as per the asymptotic normality
of Nodewise Loreg or Nodewise Lasso, where
$(\mb{M}_1)_{ij}=(\mb{M}_1)_{ji}$ remain unchanged  
if $(i,j)\in S_L(\mb{M}_3)$ with rejected $H_{ij,0}$ or if $i=j$,
and all other entries $(i,j)\in [p]\times [p]$ are set to zero.
We conduct multiple testing using the false discovery rate (FDR) control method,
AdaptZ~\citep{sun2007oracle},
which under certain conditions is optimal in minimizing the false nondiscovery rate while controlling FDR.
The nominal FDR level is set to $0.05$. 
We consider the four thresholded estimators
$\mathcal{T}(\widehat{\mb{\Omega}}^{\text{S}}|Z_0(\widehat{\mb{\Omega}}^\text{US}),S_L(\widehat{\mb{\Omega}}^\text{S}))$, 
%{\color{black}($\widehat{\mb{\Omega}}_1^{(2)}$)},
$\mathcal{T}(\widehat{\mb{\Omega}}^{\text{S}}|Z_0(\widehat{\mb{T}}),S_L(\widehat{\mb{\Omega}}^{\text{S}}))$,  
%{\color{black}($\widehat{\mb{\Omega}}_2^{(2)}$)},
$\mathcal{T}(\widehat{\mb{T}}|Z_0(\widehat{\mb{T}}),S_L(\widehat{\mb{\Omega}}^{\text{S}}))$,  
%{\color{black}($\widehat{\mb{\Omega}}_3^{(2)}$)},
and 
$\mathcal{T}(\widehat{\mb{T}}|Z_0(\widehat{\mb{T}}),S_L(\widehat{\mb{T}}))$. 
%{\color{black}($\widehat{\mb{\Omega}}_4^{(2)}$)}.
The asymptotic variances 
are estimated
 for 
the undesparsified Nodewise Loreg estimator
using 
\eqref{sigma_i,A_j, general} 
for sub-Gaussian cases
and \eqref{sigma_i,A_j, Gaussian}
for Gaussian cases,
and
for desparsified Nodewise Loreg
and desparsified Nodewise Lasso estimators 
using \eqref{sigma_ij, general} for sub-Gaussian cases and \eqref{sigma_ij, Gaussian}
for Gaussian cases.

\subsection{Results on matrix norm losses}

We evaluate matrix losses using \(\|\widehat{\mathbf{\Omega}} - \mathbf{\Omega}\|\) in the \(L_1\), spectral, Frobenius, and max norms for each method. Results based on 100 replications are summarized in Table~\ref{tab: norm loss, band Gaussian} and Tables~\ref{tab: norm loss, other Gaussian} and \ref{tab: norm loss, subGaussian} in the Supplementary Materials. 

The undesparsified Nodewise Loreg estimator \(\widehat{\mathbf{\Omega}}^{\text{S}}\) substantially outperforms the undesparsified Nodewise Lasso's \(\widehat{\mathbf{\Omega}}^{\text{S}}\) across all settings and norms, except for the max norm in Gaussian band graph settings with \(n=200\) and  in the sub-Gaussian band graph setting with \((n, p) = (200, 400)\). In these exceptions, undesparsified Nodewise Loreg surpasses
when \(n\) increases to 400.

The desparsified Nodewise Loreg estimator \(\widehat{\mathbf{T}}\) significantly underperforms its undesparsified counterpart \(\widehat{\mathbf{\Omega}}^{\text{S}}\) in the \(L_1\), spectral, and Frobenius norms, but shows comparable results in the max norm. Similar trends are observed for Nodewise Lasso, except for improvements in the spectral norm under cluster graph settings and in the max norm. The desparsified estimator \(\widehat{\mathbf{T}}\) is not sparse, leading to error accumulation over many zero entries, making it less practical.

Thresholding \(\widehat{\mathbf{T}}\) via multiple testing for Nodewise Loreg is also not recommended, since 
$\mathcal{T}(\widehat{\mb{T}}|Z_0(\widehat{\mb{T}}),S_L(\widehat{\mb{\Omega}}^{\text{S}}))$ and 
$\mathcal{T}(\widehat{\mb{T}}|Z_0(\widehat{\mb{T}}),S_L(\widehat{\mb{T}}))$ 
exhibit significantly inferior performance 
compared to  
$\widehat{\mb{\Omega}}^{\text{S}}$.
Thresholding \(\widehat{\mathbf{\Omega}}^{\text{S}}\) based on the asymptotic normality of   \(\widehat{\mathbf{\Omega}}^{\text{US}}\) or \(\widehat{\mathbf{T}}\) 
also shows no significant improvement in matrix norm losses,
as 
$\widehat{\mb{\Omega}}^{\text{S}}$,
$\mathcal{T}(\widehat{\mb{\Omega}}^{\text{S}}|Z_0(\widehat{\mb{\Omega}}^\text{US}),S_L(\widehat{\mb{\Omega}}^\text{S}))$, and 
$\mathcal{T}(\widehat{\mb{\Omega}}^{\text{S}}|Z_0(\widehat{\mb{T}}),S_L(\widehat{\mb{\Omega}}^{\text{S}}))$ have similarly good results.
This is due to their similar performance in support recovery, as shown in the next subsection.

Compared
with Nodewise Lasso's best-performing estimator 
$\mathcal{T}(\widehat{\mb{T}}|Z_0(\widehat{\mb{T}}),S_L(\widehat{\mb{T}}))$,
Nodewise Loreg's 
$\mathcal{T}(\widehat{\mb{\Omega}}^{\text{S}}|Z_0(\widehat{\mb{\Omega}}^\text{US}),S_L(\widehat{\mb{\Omega}}^\text{S}))$
shows superior performance 
in the $L_1$, spectral, and Frobenius norms
under all band and cluster graph settings,
as well as nearly all random and hub graph settings with $n=400$, %(except the spectral norm under the Gaussian hub graph when $p=400,n=400$)
achieves better and comparable max norm losses under all graph settings when $n=400$,
and generally has small differences under the other settings. 
Both estimators generally have significant smaller matrix norm losses than CLIME and GLasso estimators.

 \begin{table}[th!]
 \renewcommand\arraystretch{1.3}
 \begin{center}
\resizebox{\textwidth}{!}{
\begin{tabular}{cclcccc}
\toprule
Gaussian &   &  & $L_1$ norm& Spectral norm &Frobenius norm& Max norm  \\
Graph&   $p$& Method & $n=200/n=400$& $n=200/n=400$ &$n=200/n=400$& $n=200/n=400$  \\
\midrule

\multirow{26.5}{*}{Band}&\multirow{13}{*}{200} &
$L_0{:}~ \widehat{\mb{\Omega}}^{\text{S}}$& 1.168(0.192)/{\bf 0.637(0.096)}  & 0.834(0.154)/{\bf 0.453(0.075)} & 3.299(0.180)/1.832(0.093) & 0.511(0.089)/0.269(0.045)  \\
\multirow{13}{*}{}& \multirow{13}{*}{}& $L_0{:}~ \mathcal{T}(\widehat{\mb{\Omega}}^{\text{S}}|Z_0(\widehat{\mb{\Omega}}^\text{US}),S_L(\widehat{\mb{\Omega}}^\text{S}))$& 1.163(0.193)/0.637(0.097) & {\bf 0.832(0.154)/0.453(0.075)} &{\bf  3.295(0.179)/1.830(0.092)} & 0.511(0.089)/0.272(0.045)   \\
\multirow{13}{*}{}& \multirow{13}{*}{}& $L_0{:}~  \mathcal{T}(\widehat{\mb{\Omega}}^{\text{S}}|Z_0(\widehat{\mb{T}}),S_L(\widehat{\mb{\Omega}}^{\text{S}}))$&{\bf 1.162(0.189)}/0.638(0.098) & 0.833(0.154)/0.455(0.074) & 3.303(0.181)/1.831(0.093) & 0.511(0.089)/0.274(0.045)\\
\multirow{13}{*}{}& \multirow{13}{*}{}& $L_0{:}~  \mathcal{T}(\widehat{\mb{T}}|Z_0(\widehat{\mb{T}}),S_L(\widehat{\mb{\Omega}}^{\text{S}}))$& 1.633(0.337)/0.804(0.142) & 1.100(0.205)/0.559(0.092) & 3.979(0.249)/2.097(0.123) & 0.514(0.089)/0.277(0.043) \\
\multirow{13}{*}{}& \multirow{13}{*}{}& $L_0{:}~  \mathcal{T}(\widehat{\mb{T}}|Z_0(\widehat{\mb{T}}),S_L(\widehat{\mb{T}}))$& 1.945(0.476)/0.948(0.149) & 1.179(0.221)/0.596(0.088) & 4.585(0.260)/2.442(0.150) & 0.514(0.089)/0.288(0.039) \\
\multirow{13}{*}{}& \multirow{13}{*}{}& $L_0{:}~  \widehat{\mb{T}}$& 16.566(0.992)/10.373(0.407) & 3.375(0.139)/2.211(0.088) & 15.645(0.250)/10.386(0.129) & 0.514(0.089)/0.269(0.042) \\
\multirow{13}{*}{}& \multirow{13}{*}{}& $L_1{:}~ \widehat{\mb{\Omega}}^{\text{S}}$& 2.010(0.097)/1.332(0.058) & 1.522(0.065)/1.052(0.034) & 8.248(0.164)/5.772(0.111) & 0.497(0.018)/0.338(0.017) \\ 
\multirow{13}{*}{}& \multirow{13}{*}{}& $L_1{:}~ \mathcal{T}(\widehat{\mb{\Omega}}^{\text{S}}|Z_0(\widehat{\mb{T}}),S_L(\widehat{\mb{\Omega}}^{\text{S}}))$& 2.010(0.097)/1.332(0.058) & 1.522(0.065)/1.053(0.034) & 8.248(0.164)/5.772(0.111) & 0.497(0.018)/0.338(0.017) \\ 
\multirow{13}{*}{}& \multirow{13}{*}{}& $L_1{:}~  \mathcal{T}(\widehat{\mb{T}}|Z_0(\widehat{\mb{T}}),S_L(\widehat{\mb{\Omega}}^{\text{S}}))$& 1.940(0.230)/0.835(0.084) & 1.223(0.153)/0.546(0.046) & 4.791(0.238)/2.559(0.089) & 0.487(0.041)/0.257(0.027) \\
\multirow{13}{*}{}& \multirow{13}{*}{}& $L_1{:}~  \mathcal{T}(\widehat{\mb{T}}|Z_0(\widehat{\mb{T}}),S_L(\widehat{\mb{T}}))$& 1.328(0.145)/0.673(0.084) & 0.915(0.093)/0.493(0.051) & 3.976(0.184)/2.158(0.092) &{\bf 0.445(0.044)}/0.267(0.031) \\
\multirow{13}{*}{}& \multirow{13}{*}{}& $L_1{:}~  \widehat{\mb{T}}$& 10.455(0.330)/7.770(0.234) & 2.207(0.060)/1.598(0.042) & 11.123(0.073)/8.324(0.051) &{\bf  0.445(0.044)/0.250(0.023)} \\
\multirow{13}{*}{}& \multirow{13}{*}{}&  CLIME& 1.988(0.065)/1.454(0.074) & 1.666(0.053)/1.156(0.069) & 9.986(0.354)/6.609(0.454) & 0.534(0.022)/0.382(0.026) \\
\multirow{13}{*}{}& \multirow{13}{*}{}&  GLasso& 2.157(0.022)/1.928(0.020) & 2.021(0.006)/1.768(0.007) & 12.801(0.016)/10.958(0.023) & 0.614(0.008)/0.524(0.009) \\
\cmidrule{2-7}

\multirow{13}{*}{}&\multirow{13}{*}{400}& $L_0{:}~ \widehat{\mb{\Omega}}^{\text{S}}$& 1.384(0.187)/0.707(0.104)  &{\bf 0.972(0.136)/0.508(0.079)} & 5.052(0.208)/2.652(0.088) & 0.661(0.122)/0.301(0.039)  \\
\multirow{13}{*}{}& \multirow{13}{*}{}& $L_0{:}~ \mathcal{T}(\widehat{\mb{\Omega}}^{\text{S}}|Z_0(\widehat{\mb{\Omega}}^\text{US}),S_L(\widehat{\mb{\Omega}}^\text{S}))$& 1.380(0.183)/0.706(0.104) & {\bf 0.972(0.136)/0.508(0.079)} &{\bf  5.048(0.208)/2.647(0.089) }& 0.661(0.122)/0.301(0.039)  \\
\multirow{13}{*}{}& \multirow{13}{*}{}& $L_0{:}~  \mathcal{T}(\widehat{\mb{\Omega}}^{\text{S}}|Z_0(\widehat{\mb{T}}),S_L(\widehat{\mb{\Omega}}^{\text{S}}))$&{\bf  1.379(0.185)/0.705(0.106)} & {\bf 0.972(0.136)/0.508(0.079)} & 5.053(0.206)/2.650(0.090) & 0.661(0.122)/0.302(0.039) \\
\multirow{13}{*}{}& \multirow{13}{*}{}& $L_0{:}~  \mathcal{T}(\widehat{\mb{T}}|Z_0(\widehat{\mb{T}}),S_L(\widehat{\mb{\Omega}}^{\text{S}}))$& 2.101(0.352)/0.919(0.134) & 1.358(0.203)/0.639(0.103) & 6.258(0.278)/3.081(0.118) & 0.663(0.122)/0.304(0.037) \\
\multirow{13}{*}{}& \multirow{13}{*}{}& $L_0{:}~  \mathcal{T}(\widehat{\mb{T}}|Z_0(\widehat{\mb{T}}),S_L(\widehat{\mb{T}}))$& 2.572(0.559)/1.123(0.163) & 1.439(0.211)/0.679(0.097) & 7.227(0.289)/3.611(0.138) & 0.663(0.122)/0.312(0.029) \\
\multirow{13}{*}{}& \multirow{13}{*}{}& $L_0{:}~  \widehat{\mb{T}}$& 33.504(1.755)/20.077(0.516) & 5.585(0.186)/3.516(0.111) & 31.403(0.313)/20.751(0.153) & 0.663(0.122)/0.298(0.039) \\
\multirow{13}{*}{}& \multirow{13}{*}{}& $L_1{:}~ \widehat{\mb{\Omega}}^{\text{S}}$& 2.131(0.024)/1.423(0.061) & 1.718(0.047)/1.158(0.031) & 13.247(0.191)/9.034(0.106) & 0.532(0.022)/0.363(0.018) \\ 
\multirow{13}{*}{}& \multirow{13}{*}{}& $L_1{:}~ \mathcal{T}(\widehat{\mb{\Omega}}^{\text{S}}|Z_0(\widehat{\mb{T}}),S_L(\widehat{\mb{\Omega}}^{\text{S}}))$& 2.131(0.024)/1.423(0.061) & 1.718(0.047)/1.158(0.031) & 13.247(0.191)/9.034(0.106) & 0.532(0.022)/0.363(0.018) \\ 
\multirow{13}{*}{}& \multirow{13}{*}{}& $L_1{:}~  \mathcal{T}(\widehat{\mb{T}}|Z_0(\widehat{\mb{T}}),S_L(\widehat{\mb{\Omega}}^{\text{S}}))$& 2.150(0.064)/0.944(0.088) & 1.549(0.097)/0.628(0.047) & 8.771(0.356)/3.948(0.100) & 0.532(0.022)/0.290(0.020) \\
\multirow{13}{*}{}& \multirow{13}{*}{}& $L_1{:}~  \mathcal{T}(\widehat{\mb{T}}|Z_0(\widehat{\mb{T}}),S_L(\widehat{\mb{T}}))$& 1.676(0.145)/0.773(0.078) & 1.160(0.086)/0.576(0.051) & 6.984(0.232)/3.416(0.104) &{\bf 0.531(0.024)}/0.296(0.017) \\
\multirow{13}{*}{}& \multirow{13}{*}{}& $L_1{:}~  \widehat{\mb{T}}$& 19.481(0.582)/14.731(0.302) & 3.471(0.080)/2.417(0.042) & 20.929(0.101)/16.087(0.061) &{\bf  0.531(0.024)/0.280(0.021)} \\
\multirow{13}{*}{}& \multirow{13}{*}{}&  CLIME& 1.999(0.055)/1.645(0.060) & 1.674(0.021)/1.358(0.021) & 14.186(0.091)/11.221(0.080) & 0.546(0.015)/0.443(0.015) \\
\multirow{13}{*}{}& \multirow{13}{*}{}&  GLasso& 2.201(0.021)/2.009(0.024) & 2.025(0.005)/1.848(0.067) & 18.130(0.018)/16.286(0.691) & 0.620(0.008)/0.550(0.022) \\
\bottomrule
\end{tabular}}
\caption{Average (standard deviation) of each matrix norm loss over 100 simulation replications under Gaussian band graph settings.
The results for other (sub-)Gaussian graph settings are given in Tables~\ref{tab: norm loss, other Gaussian} and \ref{tab: norm loss, subGaussian} in Supplementary Materials.}
\label{tab: norm loss, band Gaussian}
\end{center}
 \end{table}

\subsection{Results on support recovery}
   To compare the support recovery performance of estimators on off-diagonal entries of $\mb{\Omega}$, we adopt the four evaluation metrics,  precision, sensitivity, specificity, and Matthews correlation coefficient (MCC). They  are defined as follows:
\begin{align*}
&\text{Precision}=\frac{\text{TP}}{\text{TP}+\text{FP}},
\qquad\text{Sensitivity}=\frac{\text{TP}}{\text{TP}+\text{FN}},
\qquad\text{Specificity}=\frac{\text{TN}}{\text{FP}+\text{TN}}, \\
&\text{MCC}=\frac{\text{TP}\times \text{TN}-\text{FP}\times \text{FN}}{\sqrt{(\text{TP}+\text{FP})(\text{TP}+\text{FN})(\text{TN}+\text{FP})(\text{TN}+\text{FN})}},
\end{align*}
 where TP is the number of true positives, TN is the number of true negatives, FP  is  the number of false positives, and FN  is  the number of false negatives.
Table~\ref{tab: supp recover, Gaussian band} and Tables~\ref{tab: supp recover, other Gaussian} and \ref{tab: supp recover, subGaussian} in the Supplementary Materials report the support recovery results based on 100 replications. 
The desparsified estimator \(\widehat{\mathbf{T}}\) for both Nodewise Loreg and Nodewise Lasso is excluded due to its non-sparsity. 

The four Nodewise Loreg estimators, \(\widehat{\mathbf{\Omega}}^{\text{S}}\), \(\mathcal{T}(\widehat{\mathbf{\Omega}}^{\text{S}}|Z_0(\widehat{\mathbf{\Omega}}^\text{US}),S_L(\widehat{\mathbf{\Omega}}^\text{S}))\), \(\mathcal{T}(\widehat{\mathbf{\Omega}}^{\text{S}}|Z_0(\widehat{\mathbf{T}}),S_L(\widehat{\mathbf{\Omega}}^{\text{S}}))\), and \(\mathcal{T}(\widehat{\mathbf{T}}|Z_0(\widehat{\mathbf{T}}),S_L(\widehat{\mathbf{\Omega}}^{\text{S}}))\), show similar performance with slight differences. This similarity is due to \(\widehat{\mathbf{\Omega}}^{\text{S}}\) already having high support recovery, leaving little room for improvement, 
or because
the 
strong assumption \(\max_{j \in [p]} s_j (\log p) / \sqrt{n} = o(1)\), 
required by the asymptotic normalities of
\(\widehat{\mathbf{\Omega}}^{\text{US}}\) and \(\widehat{\mathbf{T}}\) (Theorems~\ref{Corollary: AN} and \ref{thm: normality 2}),
is not ideally met for small \(n\) relative to \(p\).
The other Nodewise Loreg estimator,  \(\mathcal{T}(\widehat{\mathbf{T}}|Z_0(\widehat{\mathbf{T}}),S_L(\widehat{\mathbf{T}}))\), 
which is purely based on the desparisified estimator 
$\widehat{\mb{T}}$,
underperforms the four aforementioned estimators in all settings, except for the cluster graph settings with \(n=200\), where it shows higher MCC and sensitivity as well as lower precision, but is surpassed by the four when \(n\) increases to 400. These results suggest that the desparsified Nodewise Loreg estimator \(\widehat{\mathbf{T}}\) does not outperform its undesparsified counterpart \(\widehat{\mathbf{\Omega}}^{\text{S}}\), despite requiring weaker assumptions for asymptotic normality.

The three Nodewise Lasso estimators,  \(\widehat{\mathbf{\Omega}}^{\text{S}}\), \(\mathcal{T}(\widehat{\mathbf{\Omega}}^{\text{S}}|Z_0(\widehat{\mathbf{T}}),S_L(\widehat{\mathbf{\Omega}}^{\text{S}}))\), and  \(\mathcal{T}(\widehat{\mathbf{T}}|Z_0(\widehat{\mathbf{T}}),S_L(\widehat{\mathbf{\Omega}}^{\text{S}}))\), are inferior to their Nodewise Loreg counterparts, with significantly lower MCC. The other Nodewise Lasso estimator \(\mathcal{T}(\widehat{\mathbf{T}}|Z_0(\widehat{\mathbf{T}}),S_L(\widehat{\mathbf{T}}))\) underperforms the first four Nodewise Loreg estimators by up to 0.152 lower average MCC, except for random graph settings with \(n=200\), where it is still surpassed by or comparable to the  four when \(n\) increases to 400.
The CLIME and GLasso estimators perform worse than the first four Nodewise Loreg estimators, with significantly lower MCC in all settings. Overall, the first four Nodewise Loreg estimators exhibit the best performance in support recovery.

 \begin{table}[th!]
 \renewcommand\arraystretch{1.3}
 \begin{center}
\resizebox{\textwidth}{!}{
\begin{tabular}{cclcccc}
\toprule
Gaussian &   &  & Precision& Sensitivity &Specificity& MCC  \\
Graph&   $p$& Method & $n=200/n=400$& $n=200/n=400$ &$n=200/n=400$& $n=200/n=400$  \\
\midrule

\multirow{22}{*}{Band}&\multirow{10.5}{*}{200} &
$L_0{:}~ \widehat{\mb{\Omega}}^{\text{S}}$& 0.983(0.007)/0.992(0.005)  &{\bf 0.979(0.008)/1.000(0.001)} &{\bf 1.000(0.000)/1.000(0.000)} & 0.980(0.006)/0.996(0.003) \\
\multirow{10.5}{*}{}& \multirow{10.5}{*}{}& $L_0{:}~ \mathcal{T}(\widehat{\mb{\Omega}}^{\text{S}}|Z_0(\widehat{\mb{\Omega}}^\text{US}),S_L(\widehat{\mb{\Omega}}^\text{S}))$&{\bf 0.985(0.007)/0.994(0.005)}  &{\bf 0.979(0.008)/1.000(0.001)} &{\bf 1.000(0.000)/1.000(0.000)} &{\bf 0.981(0.006)/0.997(0.002)}  \\
\multirow{10.5}{*}{}& \multirow{10.5}{*}{}& $L_0{:}~  \mathcal{T}(\widehat{\mb{\Omega}}^{\text{S}}|Z_0(\widehat{\mb{T}}),S_L(\widehat{\mb{\Omega}}^{\text{S}}))$& 0.984(0.006)/{\bf 0.994(0.005)}  & 0.978(0.008)/0.999(0.001) &{\bf 1.000(0.000)/1.000(0.000)} &{\bf 0.981(0.006)/0.997(0.002)} \\
\multirow{10.5}{*}{}& \multirow{10.5}{*}{}& $L_0{:}~  \mathcal{T}(\widehat{\mb{T}}|Z_0(\widehat{\mb{T}}),S_L(\widehat{\mb{\Omega}}^{\text{S}}))$& 0.984(0.006)/{\bf 0.994(0.005)}  & 0.978(0.008)/0.999(0.001) &{\bf 1.000(0.000)/1.000(0.000)} &{\bf 0.981(0.006)/0.997(0.002)} \\
\multirow{10.5}{*}{}& \multirow{10.5}{*}{}& $L_0{:}~  \mathcal{T}(\widehat{\mb{T}}|Z_0(\widehat{\mb{T}}),S_L(\widehat{\mb{T}}))$& 0.931(0.015)/0.931(0.010) & 0.949(0.012)/0.998(0.002) & 0.999(0.000)/0.998(0.000) & 0.938(0.009)/0.963(0.006) \\
\multirow{10.5}{*}{}& \multirow{10.5}{*}{}& $L_1{:}~ \widehat{\mb{\Omega}}^{\text{S}}$& 0.780(0.016)/0.810(0.014) & 0.928(0.016)/{\bf 1.000(0.001)} & 0.995(0.000)/0.995(0.000) & 0.847(0.012)/0.897(0.008) \\ 
\multirow{10.5}{*}{}& \multirow{10.5}{*}{}& $L_1{:}~ \mathcal{T}(\widehat{\mb{\Omega}}^{\text{S}}|Z_0(\widehat{\mb{T}}),S_L(\widehat{\mb{\Omega}}^{\text{S}}))$& 0.782(0.016)/0.811(0.014) & 0.928(0.016)/{\bf 1.000(0.001)} & 0.995(0.000)/0.995(0.000) & 0.848(0.012)/0.898(0.008) \\ 
\multirow{10.5}{*}{}& \multirow{10.5}{*}{}& $L_1{:}~  \mathcal{T}(\widehat{\mb{T}}|Z_0(\widehat{\mb{T}}),S_L(\widehat{\mb{\Omega}}^{\text{S}}))$& 0.782(0.016)/0.811(0.014) & 0.928(0.016)/{\bf 1.000(0.001)} & 0.995(0.000)/0.995(0.000) & 0.848(0.012)/0.898(0.008) \\
\multirow{10.5}{*}{}& \multirow{10.5}{*}{}& $L_1{:}~  \mathcal{T}(\widehat{\mb{T}}|Z_0(\widehat{\mb{T}}),S_L(\widehat{\mb{T}}))$& 0.963(0.009)/0.948(0.002) & 0.966(0.009)/0.999(0.001) & 0.999(0.000)/0.999(0.000) & 0.964(0.007)/0.973(0.001) \\
\multirow{10.5}{*}{}& \multirow{10.5}{*}{}&  CLIME& 0.567(0.043)/0.373(0.077) &	0.963(0.011)/{\bf 1.000(0.001)} &	0.985(0.005)/0.964(0.007) &	0.732(0.032)/0.598(0.058) \\
\multirow{10.5}{*}{}& \multirow{10.5}{*}{}&  GLasso& 0.497(0.015)/0.377(0.011) & 0.932(0.014)/0.999(0.001) & 0.981(0.001)/0.966(0.002) & 0.672(0.011)/0.603(0.009) \\
\cmidrule{2-7}

\multirow{10.5}{*}{}&\multirow{10.5}{*}{400}& $L_0{:}~ \widehat{\mb{\Omega}}^{\text{S}}$& 0.984(0.005)/0.993(0.003)  & 0.969(0.007)/{\bf 1.000(0.001)} &{\bf 1.000(0.000)/1.000(0.000)} & 0.976(0.005)/0.996(0.001)  \\
\multirow{10.5}{*}{}& \multirow{10.5}{*}{}& $L_0{:}~ \mathcal{T}(\widehat{\mb{\Omega}}^{\text{S}}|Z_0(\widehat{\mb{\Omega}}^\text{US}),S_L(\widehat{\mb{\Omega}}^\text{S}))$&{\bf 0.985(0.005)/0.994(0.003)}  & 0.969(0.007)/{\bf 1.000(0.001)} &{\bf 1.000(0.000)/1.000(0.000)} &{\bf 0.977(0.005)/0.997(0.001)} \\
\multirow{10.5}{*}{}& \multirow{10.5}{*}{}& $L_0{:}~  \mathcal{T}(\widehat{\mb{\Omega}}^{\text{S}}|Z_0(\widehat{\mb{T}}),S_L(\widehat{\mb{\Omega}}^{\text{S}}))$&{\bf 0.985(0.005)/0.994(0.003)}  & 0.969(0.007)/{\bf 1.000(0.001)} &{\bf 1.000(0.000)/1.000(0.000)} &{\bf 0.977(0.005)/0.997(0.001)} \\
\multirow{10.5}{*}{}& \multirow{10.5}{*}{}& $L_0{:}~  \mathcal{T}(\widehat{\mb{T}}|Z_0(\widehat{\mb{T}}),S_L(\widehat{\mb{\Omega}}^{\text{S}}))$&{\bf 0.985(0.005)/0.994(0.003)}  & 0.969(0.007)/{\bf 1.000(0.001)} &{\bf 1.000(0.000)/1.000(0.000)} &{\bf 0.977(0.005)/0.997(0.001)} \\
\multirow{10.5}{*}{}& \multirow{10.5}{*}{}& $L_0{:}~  \mathcal{T}(\widehat{\mb{T}}|Z_0(\widehat{\mb{T}}),S_L(\widehat{\mb{T}}))$& 0.930(0.011)/0.936(0.007) & 0.930(0.009)/0.997(0.002) & 0.999(0.000)/0.999(0.000) & 0.929(0.007)/0.966(0.003) \\
\multirow{10.5}{*}{}& \multirow{10.5}{*}{}& $L_1{:}~ \widehat{\mb{\Omega}}^{\text{S}}$& 0.797(0.011)/0.822(0.011) & 0.827(0.022)/0.999(0.001) & 0.998(0.000)/0.998(0.000) & 0.810(0.012)/0.905(0.006) \\ 
\multirow{10.5}{*}{}& \multirow{10.5}{*}{}& $L_1{:}~ \mathcal{T}(\widehat{\mb{\Omega}}^{\text{S}}|Z_0(\widehat{\mb{T}}),S_L(\widehat{\mb{\Omega}}^{\text{S}}))$& 0.798(0.011)/0.823(0.011) & 0.827(0.022)/0.999(0.001) & 0.998(0.000)/0.998(0.000) & 0.810(0.012)/0.906(0.006) \\ 
\multirow{10.5}{*}{}& \multirow{10.5}{*}{}& $L_1{:}~  \mathcal{T}(\widehat{\mb{T}}|Z_0(\widehat{\mb{T}}),S_L(\widehat{\mb{\Omega}}^{\text{S}}))$& 0.798(0.011)/0.823(0.011) & 0.827(0.022)/0.999(0.001) & 0.998(0.000)/0.998(0.000) & 0.810(0.012)/0.906(0.006) \\
\multirow{10.5}{*}{}& \multirow{10.5}{*}{}& $L_1{:}~  \mathcal{T}(\widehat{\mb{T}}|Z_0(\widehat{\mb{T}}),S_L(\widehat{\mb{T}}))$& 0.957(0.008)/0.948(0.001) & 0.945(0.010)/0.998(0.002) &{\bf 1.000(0.000)}/0.999(0.000) & 0.951(0.008)/0.973(0.001) \\
\multirow{10.5}{*}{}& \multirow{10.5}{*}{}&  CLIME& 0.482(0.007)/0.495(0.009) &{\bf 0.970(0.006)/1.000(0.001)} & 0.989(0.000)/0.990(0.000) & 0.680(0.006)/0.700(0.006) \\
\multirow{10.5}{*}{}& \multirow{10.5}{*}{}&  GLasso& 0.391(0.010)/0.402(0.119) & 0.932(0.010)/0.998(0.002) & 0.985(0.001)/0.982(0.009) & 0.599(0.008)/0.621(0.098) \\
\bottomrule
\end{tabular}}
\caption{Average (standard deviation) of each support-recovery metric over 100 simulation replications under Gaussian band graph settings.
The results for other (sub-)Gaussian graph settings are given in Tables~\ref{tab: supp recover, other Gaussian} and
\ref{tab: supp recover, subGaussian} in Supplementary Materials.}
\label{tab: supp recover, Gaussian band}
\end{center}
 \end{table}

\subsection{Results on asymptotic normality}
    
We evaluate the asymptotic normality of Nodewise Loreg in comparison to Nodewise Lasso.
We consider the unsymmetrized Nodewise Loreg estimator $\widehat{\mb{\Omega}}^{\text{US}}$
and the desparsified estimator $\widehat{\mb{T}}$ respectively from Nodewise Loreg and Nodewise Lasso.
Denote $S_{\mb{\Omega}}=\supp(\mb{\Omega})$, $S_{\mb{\Omega}}^c=[p]\times [p] \setminus S_{\mb{\Omega}}$,
and
$\widehat{S}_{\mb{\Omega}}=\bigcap_{m=1}^{M}\supp(\widehat{\mb{\Omega}}^{\text{US},(m)}_{L_0})$, where $\widehat{\mb{\Omega}}^{\text{US},(m)}_{L_0}$ is the unsymmetrized estimate of $\mb{\Omega}$ from Nodewise Loreg 
based on the data from the $m$-th replication.
We consider four evaluation metrics computed for each single entry in a given set.
The average and standard deviation of each evaluation metric  are calculated over all entries in $\widehat{S}_{\mb{\Omega}}$ for all the three estimators, 
and in $S_{\mb{\Omega}}$ and in $S_{\mb{\Omega}}^c$ only for the two desparsified estimators which are applicable there.
For entry $(i,j)$, the four evaluation metrics are defined as follows,
where $\widehat{\sigma}_{\mb{\Omega}_{ij}}^{(m)}$ and $\widehat{\bm \Omega}^{(m)}$ denote the estimates of the asymptotic variance $\sigma_{\mb{\Omega}_{ij}}$ (either $\sigma_{i,\widehat{A}_j}$ or $\sigma_{ij}$ as per method) and the precision matrix $\mb{\Omega}$, respectively, obtained from 
the $m$-th replication.

\begin{itemize}[leftmargin=*, itemsep=0pt]
\item Average length of estimated $100(1-\alpha)\%$ confidence intervals:
\[
\AvgLength_{(i,j)}
=\frac{1}{M} \sum_{m=1}^{M} 2 \Phi^{-1}(1-\alpha / 2) \widehat{\sigma}_{\mb{\Omega}_{ij}}^{(m)} /\sqrt{n},
\]
which estimates the length of the corresponding true confidence interval
\[
\TrueLength_{(i,j)}
= 2 \Phi^{-1}(1-\alpha / 2)\sigma_{\mb{\Omega}_{ij}} /\sqrt{n}.
\]

\item
Coverage rate of estimated $100(1-\alpha)\%$ confidence intervals: 
\begin{align*}
\CovRate_{(i,j)}=\frac{1}{M}\sum_{m=1}^{M}I\Big(\mb{\Omega}_{ij}\in \Big[&\widehat{\mb{\Omega}}_{ij}^{(m)}-\Phi^{-1}(1-\alpha / 2) \widehat{\sigma}_{\mb{\Omega}_{ij}}^{(m)} /\sqrt{n},\\
&\widehat{\mb{\Omega}}_{ij}^{(m)}+\Phi^{-1}(1-\alpha / 2) \widehat{\sigma}_{\mb{\Omega}_{ij}}^{(m)}/\sqrt{n} \Big]\Big).
\end{align*}

\item Absolute average $Z$-score: 
\[
\AbsAvgZ_{(i,j)}=\Big|\frac{1}{M}\sum_{m=1}^{M}\sqrt{n}(\widehat{\mb{\Omega}}_{ij}^{(m)}-\mb{\Omega}_{ij})/\widehat{\sigma}_{\mb{\Omega}_{ij}}^{(m)}\Big|.
\]

\item Standard deviation of $Z$-scores: 
\[
\SDZ_{(i,j)}=\sd\Big(\big(\sqrt{n}(\widehat{\mb{\Omega}}_{ij}^{(m)}-\mb{\Omega}_{ij})/\widehat{\sigma}_{\mb{\Omega}_{ij}}^{(m)}\big)_{m=1}^{M}\Big).
\]
\end{itemize}

 \begin{table}[b!]
 \renewcommand\arraystretch{1.3}
 \begin{center}
\resizebox{\textwidth}{!}{
\begin{tabular}{cclccccc}
\toprule
Gaussian &   &  & $\TrueLength_{(i,j)\in\widehat{S}_{\mb{\Omega}}}$& $\AvgLength_{(i,j)\in\widehat{S}_{\mb{\Omega}}}$ &$\CovRate_{(i,j)\in\widehat{S}_{\mb{\Omega}}}$& $\AbsAvgZ_{(i,j)\in\widehat{S}_{\mb{\Omega}}}$ & $\SDZ_{(i,j)\in\widehat{S}_{\mb{\Omega}}}$  \\
Graph&   $p$& Method & $n=200/n=400$& $n=200/n=400$ &$n=200/n=400$& $n=200/n=400$ & $n=200/n=400$ \\
\midrule

\multirow{6.5}{*}{Band}&\multirow{3}{*}{200} &
$L_0{:}~ \widehat{\mb{\Omega}}^{\text{US}}$& {\bf 0.322(0.048)/0.209(0.038)}  & {\bf 0.347(0.054)/0.214(0.039)} &{\bf 0.904(0.032)/0.942(0.023)} & {\bf 0.445(0.158)/0.204(0.116)} &	1.088(0.094)/{\bf 1.010(0.071)} \\
\multirow{6.5}{*}{}& \multirow{3}{*}{}& $L_0{:}~ \widehat{\mb{T}}$& 0.336(0.039)/0.225(0.027)  & 0.362(0.043)/0.232(0.028) &	0.883(0.033)/0.938(0.024) &	0.631(0.103)/0.283(0.111) &	{\bf 1.075(0.085)}/1.012(0.073)  \\
\multirow{6.5}{*}{}& \multirow{3}{*}{}& $L_1{:}~ \widehat{\mb{T}}$& 0.336(0.039)/0.225(0.027) & 0.257(0.039)/0.190(0.026) &	0.627(0.097)/0.846(0.088) &	1.536(0.378)/0.673(0.407) &	1.497(0.160)/1.144(0.117) \\ 
\cmidrule{2-8}

\multirow{6.5}{*}{}&\multirow{3}{*}{400}& $L_0{:}~ \widehat{\mb{\Omega}}^{\text{US}}$& {\bf 0.322(0.047)/0.209(0.038)}  & {\bf 0.351(0.054)/0.215(0.040)} &{\bf	0.884(0.036)/0.938(0.024)} &	{\bf 0.505(0.180)/0.230(0.123)} &	1.139(0.098)/{\bf 1.024(0.073)}  \\
\multirow{6.5}{*}{}& \multirow{3}{*}{}& $L_0{:}~ \widehat{\mb{T}}$& 0.337(0.039)/0.225(0.027)  & 0.367(0.043)/0.233(0.028)	& 0.851(0.038)/0.930(0.025) &	0.754(0.127)/0.325(0.115) &	{\bf 1.130(0.089)}/1.031(0.074) \\
\multirow{6.5}{*}{}& \multirow{3}{*}{}& $L_1{:}~ \widehat{\mb{T}}$& 0.337(0.039)/0.225(0.027) & 0.241(0.038)/0.184(0.026) &	0.443(0.107)/0.787(0.129) &	2.344(0.459)/0.916(0.506) &	1.777(0.226)/1.196(0.132) \\ 
\cmidrule{1-8}

\multirow{6.5}{*}{Random}&\multirow{3}{*}{200} &
$L_0{:}~ \widehat{\mb{\Omega}}^{\text{US}}$&{\bf 0.583(0.258)/0.375(0.174)}  & 0.652(0.308)/ {\bf 0.387(0.183)} &{\bf	0.898(0.033)/0.938(0.025)} &{\bf	0.372(0.133)/0.216(0.126)} &	{\bf 1.139(0.094)}/1.026(0.078) \\
\multirow{6.5}{*}{}& \multirow{3}{*}{}& $L_0{:}~ \widehat{\mb{T}}$& 0.624(0.296)/0.433(0.197)  & {\bf 0.659(0.322)}/0.446(0.205) &{\bf	0.898(0.033)}/0.935(0.025) &	0.375(0.138)/0.259(0.114) &	1.140(0.094)/{\bf 1.026(0.074)}  \\
\multirow{6.5}{*}{}& \multirow{3}{*}{}& $L_1{:}~ \widehat{\mb{T}}$& 0.624(0.296)/0.433(0.197) & 0.516(0.208)/0.377(0.153)	& 0.616(0.169)/0.890(0.089) &	1.628(0.686)/0.528(0.391) &	1.531(0.350)/1.059(0.151) \\ 
\cmidrule{2-8}

\multirow{6.5}{*}{}&\multirow{3}{*}{400}& $L_0{:}~ \widehat{\mb{\Omega}}^{\text{US}}$& {\bf 0.585(0.200)/0.371(0.138)}  & 0.667(0.251)/ {\bf 0.384(0.146)} &{\bf 0.874(0.038)/0.936(0.026)} &	{\bf 0.402(0.137)/0.231(0.129)} &	1.212(0.110)/{\bf 1.032(0.076)}  \\
\multirow{6.5}{*}{}& \multirow{3}{*}{}& $L_0{:}~ \widehat{\mb{T}}$& 0.639(0.240)/0.441(0.145)  &  {\bf 0.678(0.261)}/0.455(0.152) &{\bf	0.874(0.038)}/0.932(0.025) &	0.406(0.140)/0.279(0.117) &	{\bf 1.211(0.111)}/1.038(0.074) \\
\multirow{6.5}{*}{}& \multirow{3}{*}{}& $L_1{:}~ \widehat{\mb{T}}$& 0.639(0.240)/0.441(0.145) & 0.509(0.153)/0.375(0.107) &	0.534(0.197)/0.865(0.123) &	2.039(0.875)/0.668(0.484) &	1.609(0.415)/1.053(0.176) \\ 
\cmidrule{1-8}

\multirow{6.5}{*}{Hub}&\multirow{3}{*}{200} &
$L_0{:}~ \widehat{\mb{\Omega}}^{\text{US}}$&{\bf 1.328(1.092)/1.030(0.707)}  & {\bf 1.406(1.212)/1.056(0.736)} &{\bf	0.932(0.027)/0.945(0.022)} &	{\bf 0.270(0.137)/0.160(0.108)} &	{\bf 1.040(0.078)/1.008(0.075)} \\
\multirow{6.5}{*}{}& \multirow{3}{*}{}& $L_0{:}~ \widehat{\mb{T}}$& 1.726(1.063)/1.245(0.662)  & 1.829(1.161)/1.276(0.688) &	0.919(0.028)/0.941(0.023) &	0.356(0.131)/0.206(0.113) &	1.060(0.080)/1.012(0.076)  \\
\multirow{6.5}{*}{}& \multirow{3}{*}{}& $L_1{:}~ \widehat{\mb{T}}$& 1.726(1.063)/1.245(0.662) & 1.333(0.692)/1.020(0.476) &	0.891(0.086)/0.920(0.063) &	0.398(0.395)/0.267(0.297) &	1.124(0.156)/1.050(0.115) \\ 
\cmidrule{2-8}

\multirow{6.5}{*}{}&\multirow{3}{*}{400}& $L_0{:}~ \widehat{\mb{\Omega}}^{\text{US}}$&{\bf 1.297(1.107)/1.028(0.708)}  &{\bf  1.379(1.246)/1.056(0.739) }&{\bf	0.924(0.031)/0.942(0.024)} &	{\bf 0.309(0.144)/0.175(0.106)} &	{\bf 1.053(0.079)/1.019(0.077)}  \\
\multirow{6.5}{*}{}& \multirow{3}{*}{}& $L_0{:}~ \widehat{\mb{T}}$& 1.717(1.090)/1.244(0.664)  & 1.831(1.202)/1.277(0.691) &	0.907(0.033)/0.936(0.025) &	0.408(0.144)/0.226(0.112) &	1.082(0.084)/1.031(0.077) \\
\multirow{6.5}{*}{}& \multirow{3}{*}{}& $L_1{:}~ \widehat{\mb{T}}$& 1.717(1.090)/1.244(0.664) & 1.289(0.675)/1.000(0.461) &	0.863(0.114)/0.909(0.079) &	0.510(0.517)/0.315(0.371) &	1.147(0.168)/1.066(0.127) \\ 
\cmidrule{1-8}

\multirow{6.5}{*}{Cluster}&\multirow{3}{*}{200} &
$L_0{:}~ \widehat{\mb{\Omega}}^{\text{US}}$&{\bf 0.739(0.355)/0.401(0.170)}  & 0.831(0.396)/{\bf 0.413(0.176)} &{\bf	0.849(0.055)/0.941(0.024)} &{\bf	0.473(0.180)/0.191(0.132) }&	{\bf 1.261(0.207)/1.016(0.070) }\\
\multirow{6.5}{*}{}& \multirow{3}{*}{}& $L_0{:}~ \widehat{\mb{T}}$& 0.789(0.373)/0.426(0.175)  &{\bf  0.847(0.413)}/0.441(0.183) &	0.847(0.053)/0.936(0.025) &	0.498(0.174)/0.265(0.123) &{\bf	1.261(0.207)}/1.020(0.070)  \\
\multirow{6.5}{*}{}& \multirow{3}{*}{}& $L_1{:}~ \widehat{\mb{T}}$& 0.789(0.373)/0.426(0.175) & 0.586(0.281)/0.338(0.138) &	0.352(0.279)/0.537(0.280) &	2.908(1.411)/1.967(1.251) &	1.631(0.312)/1.306(0.223) \\ 
\cmidrule{2-8}

\multirow{6.5}{*}{}&\multirow{3}{*}{400}& $L_0{:}~ \widehat{\mb{\Omega}}^{\text{US}}$&{\bf 0.600(0.246)/0.359(0.147)}  & 0.695(0.283)/{\bf 0.371(0.154)} &{\bf	0.809(0.072)/0.938(0.026)} &	{\bf 0.391(0.248)/0.217(0.146)}  &	{\bf 1.431(0.233)/1.021(0.076)} \\
\multirow{6.5}{*}{}& \multirow{3}{*}{}& $L_0{:}~ \widehat{\mb{T}}$& 0.662(0.257)/0.381(0.151)  &{\bf  0.697(0.289)}/0.395(0.158) &{\bf	0.809(0.072)}/0.932(0.027) &	0.392(0.249)/0.294(0.126) &	1.431(0.234)/1.029(0.077) \\
\multirow{6.5}{*}{}& \multirow{3}{*}{}& $L_1{:}~ \widehat{\mb{T}}$& 0.662(0.257)/0.381(0.151) & 0.485(0.183)/0.301(0.118) &	0.228(0.200)/0.527(0.297) &	3.513(1.144)/1.997(1.308) &	1.654(0.288)/1.280(0.232) \\ 
\bottomrule
\end{tabular}}
\caption{Average (standard deviation) of each asymptotic-normality metric over all entries in $\widehat{S}_{\mb{\Omega}}$ under Gaussian settings based on 100 simulation replications.
The results for other \mbox{(sub-)Gaussian} settings on $\widehat{S}_{\mb{\Omega}}$,  $S_{\mb{\Omega}}$ and $S_{\mb{\Omega}}^c$ are given in Tables~\ref{tab: AN S_omega, Gaussian}--\ref{tab: AN S_omega complement, subGaussian} in Supplementary \mbox{Materials.}}
\label{tab: AN S_omega hat, Gaussian}
\end{center}
 \end{table}

Table~\ref{tab: AN S_omega hat, Gaussian} and Tables~\ref{tab: AN S_omega, Gaussian}--\ref{tab: AN S_omega complement, subGaussian} in the Supplementary Materials present the results on asymptotic normality using the four evaluation metrics, based on \(M = 100\) replications and \(\alpha = 0.05\) for constructing confidence intervals.

Specifically, Table~\ref{tab: AN S_omega hat, Gaussian} and Table~\ref{tab: AN S_omega hat, subGaussian} show the results on \(\widehat{S}_{\mathbf{\Omega}}\) under Gaussian and sub-Gaussian settings, respectively.
We observe that the true length of the 95\% confidence interval, \(\TrueLength_{(i,j)}\), for the unsymmetrized Nodewise Loreg estimator \(\widehat{\mathbf{\Omega}}^{\text{US}}\) is, on average, smaller on \(\widehat{S}_{\mathbf{\Omega}}\) than for the desparsified estimator \(\widehat{\mathbf{T}}\) of both Nodewise Loreg and Nodewise Lasso. This result is consistent with Theorem~\ref{sigma_ij compare} for multivariate Gaussian distributions, though the theorem does not universally apply to sub-Gaussian scenarios.

The average length of estimated confidence intervals, \(\AvgLength_{(i,j)}\), from both Nodewise Loreg estimators \(\widehat{\mathbf{\Omega}}^{\text{US}}\) and \(\widehat{\mathbf{T}}\) more accurately approximates the true length, \(\TrueLength_{(i,j)}\), across \(\widehat{S}_{\mathbf{\Omega}}\) compared to the Nodewise Lasso estimator \(\widehat{\mathbf{T}}\). This indicates that our Nodewise Loreg estimators for asymptotic variances \(\sigma_{i,\widehat{A}_j}\) and \(\sigma_{ij}\) are generally closer to their true values on \(\widehat{S}_{\mathbf{\Omega}}\) than the Nodewise Lasso estimator for \(\sigma_{ij}\).

The unsymmetrized Nodewise Loreg estimator \(\widehat{\mathbf{\Omega}}^{\text{US}}\) outperforms the other two estimators in the coverage rate of estimated 95\% confidence intervals, \(\CovRate_{(i,j)}\), on \(\widehat{S}_{\mathbf{\Omega}}\), aligning more closely with the nominal value of 0.95. The desparsified Nodewise Loreg estimator \(\widehat{\mathbf{T}}\) ranks second, while the desparsified Nodewise Lasso estimator \(\widehat{\mathbf{T}}\) significantly underperforms, particularly in cluster graph settings, trailing by over 0.4 on average compared to \(\widehat{\mathbf{\Omega}}^{\text{US}}\).

Regarding the \(Z\)-score, the unsymmetrized Nodewise Loreg estimator \(\widehat{\mathbf{\Omega}}^{\text{US}}\) performs the best in all settings in terms of the absolute average \(Z\)-score \(\AbsAvgZ_{(i,j)}\) being closer to zero on \(\widehat{S}_{\mathbf{\Omega}}\), leading by a large margin in most settings compared to Nodewise Loreg's desparsified estimator \(\widehat{\mathbf{T}}\) and even more significantly against Nodewise Lasso's \(\widehat{\mathbf{T}}\) in all settings. This result suggests that debiasing Nodewise Loreg's \(\widehat{\mathbf{\Omega}}^{\text{US}}\) may be unnecessary. Additionally, with the standard deviation of \(Z\)-scores, \(\SDZ_{(i,j)}\), being closest to one or nearly the best, the \(Z\)-score of Nodewise Loreg's \(\widehat{\mathbf{\Omega}}^{\text{US}}\) more closely follows the standard Gaussian distribution.

Comparative analysis of Nodewise Loreg's desparsified estimator \(\widehat{\mathbf{T}}\) and Nodewise Lasso's \(\widehat{\mathbf{T}}\) based on the results on \(S_{\mathbf{\Omega}}\) and \(S_{\mathbf{\Omega}}^c\) given in Tables~\ref{tab: AN S_omega, Gaussian}, \ref{tab: AN S_omega complement, Gaussian}, \ref{tab: AN S_omega, subGaussian}, and \ref{tab: AN S_omega complement, subGaussian} shows the former's superior performance in most settings.

%Gaussian band
  \begin{figure}[bh!]
  \caption*{\scriptsize $n=200, p=200$}
  \vspace{-0.43cm}
 \begin{minipage}{0.48\linewidth}
    \begin{minipage}{0.32\linewidth}
        \centering
        \includegraphics[width=\textwidth]{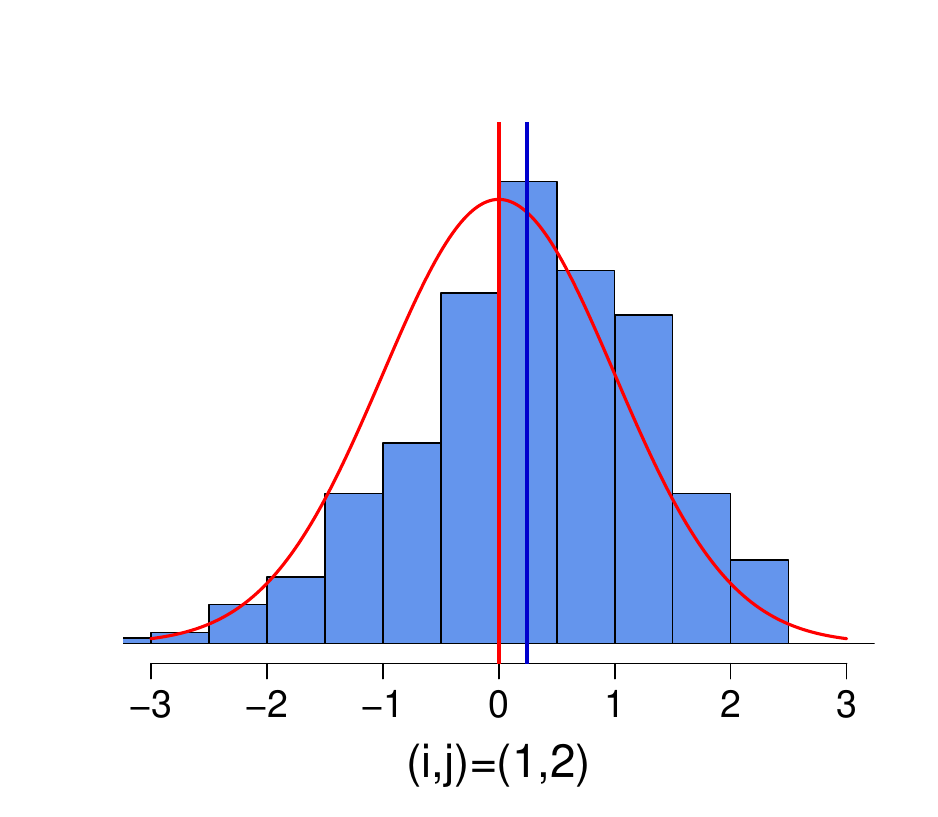}
  %          \caption*{\scriptsize $L_0{:}~ \widehat{\mb{\Omega}}^{\text{US}}$}
    \end{minipage}
    \begin{minipage}{0.32\linewidth}
        \centering
        \includegraphics[width=\textwidth]{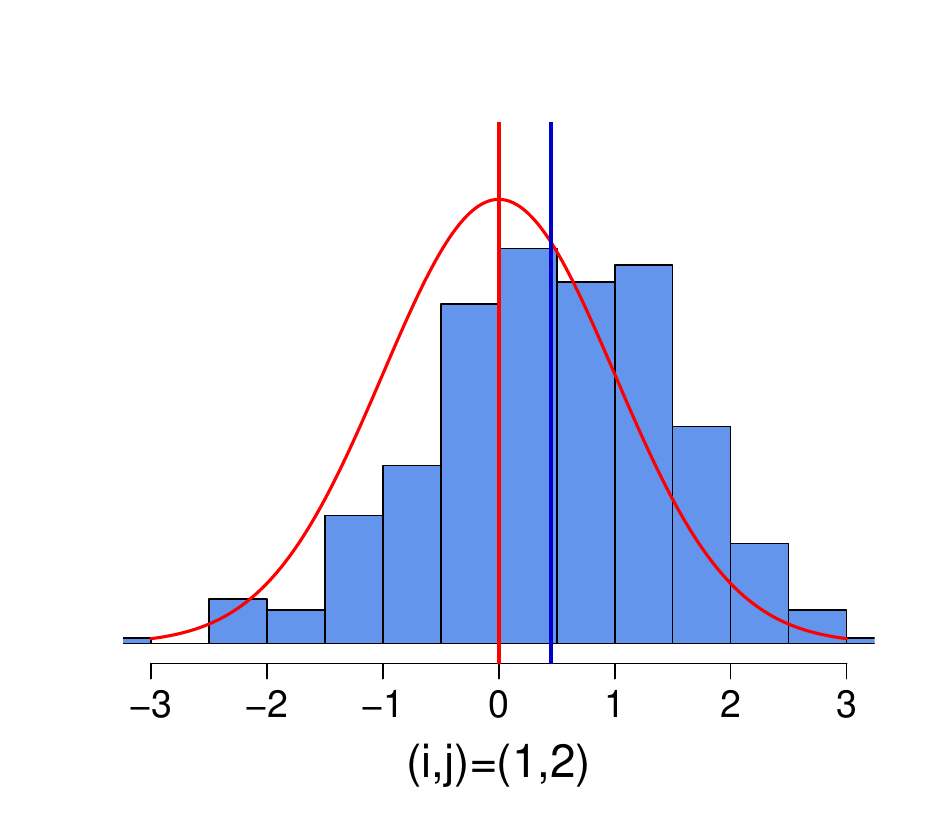}
 %                   \caption*{\scriptsize $L_0{:}~ \widehat{\mb{T}}$}
    \end{minipage}
    \begin{minipage}{0.32\linewidth}
        \centering
        \includegraphics[width=\textwidth]{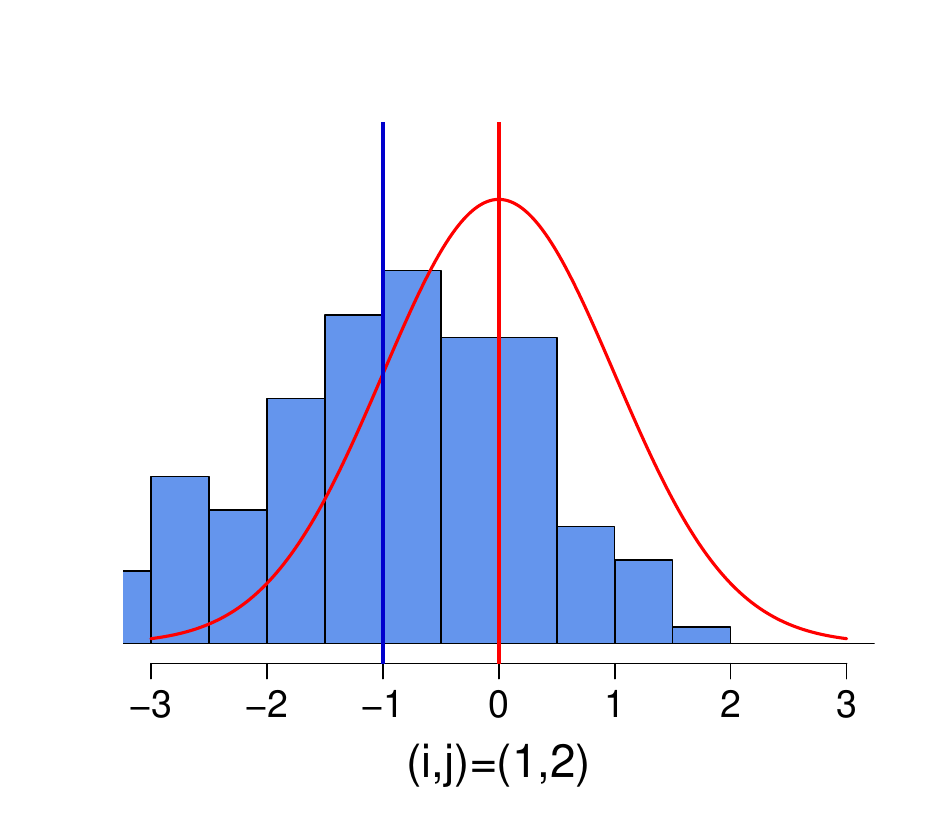}
 %          \caption*{\scriptsize $L_1{:}~ \widehat{\mb{T}}$}
    \end{minipage}
 \end{minipage}
 \hfill
 \begin{minipage}{0.48\linewidth}
    \begin{minipage}{0.32\linewidth}
        \centering
        \includegraphics[width=\textwidth]{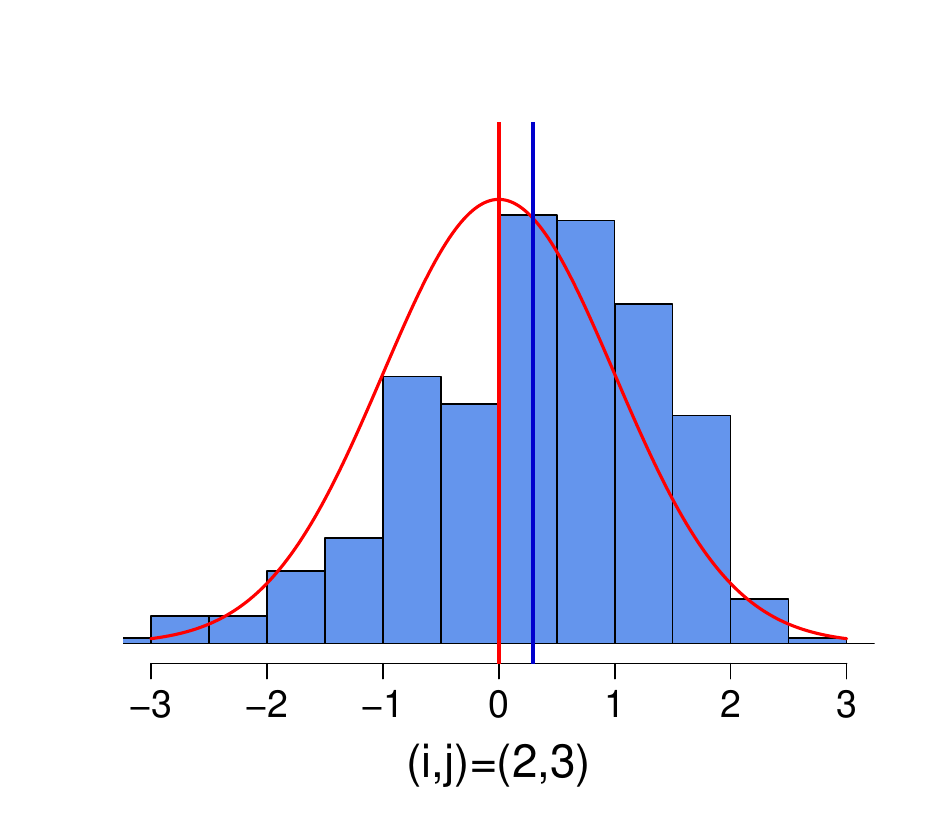}
%    \caption*{\scriptsize $L_0{:}~ \widehat{\mb{\Omega}}^{\text{US}}$}
    \end{minipage}
    \begin{minipage}{0.32\linewidth}
        \centering
        \includegraphics[width=\textwidth]{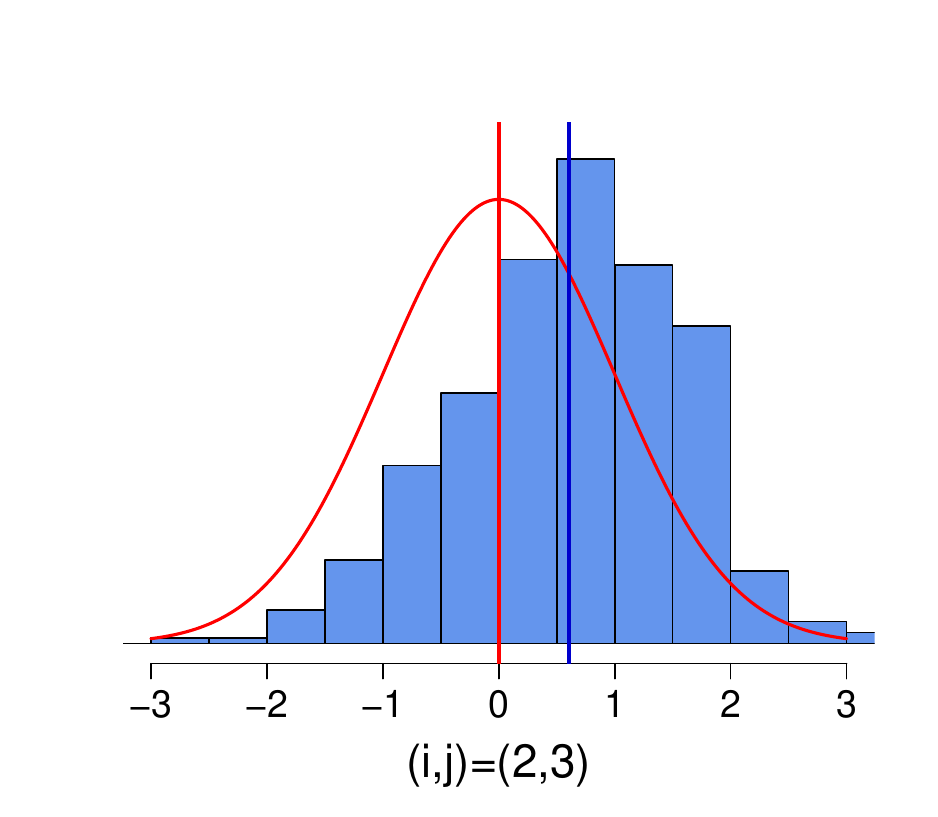}
%           \caption*{\scriptsize $L_0{:}~ \widehat{\mb{T}}$}
    \end{minipage}
    \begin{minipage}{0.32\linewidth}
        \centering
        \includegraphics[width=\textwidth]{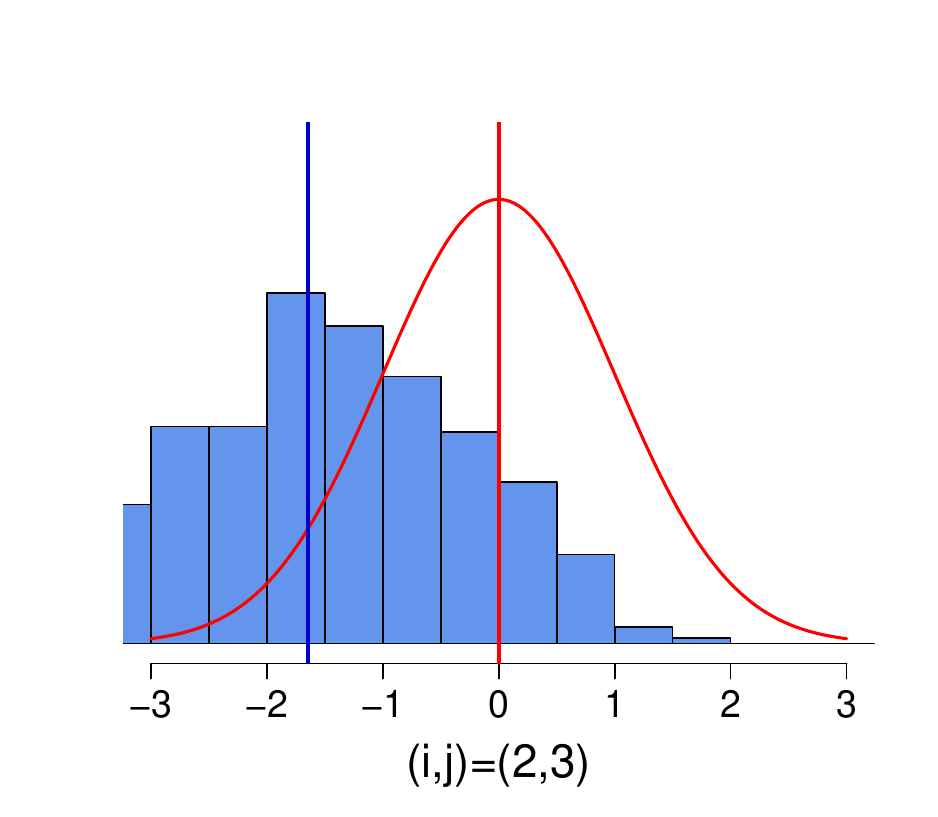}
%           \caption*{\scriptsize $L_1{:}~ \widehat{\mb{T}}$}
    \end{minipage}   
 \end{minipage}

  \caption*{\scriptsize $n=400, p=200$}
    \vspace{-0.43cm}
 \begin{minipage}{0.48\linewidth}
    \begin{minipage}{0.32\linewidth}
        \centering
        \includegraphics[width=\textwidth]{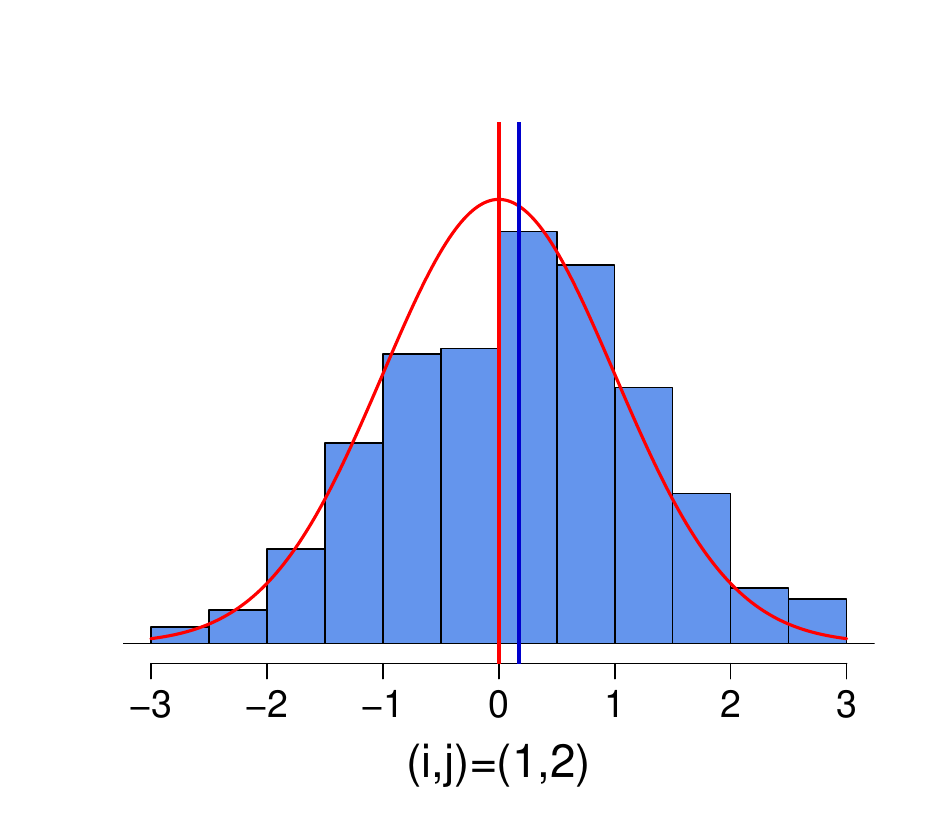}
 %           \caption*{\scriptsize $L_0{:}~ \widehat{\mb{\Omega}}^{\text{US}}$}
    \end{minipage}
    \begin{minipage}{0.32\linewidth}
        \centering
        \includegraphics[width=\textwidth]{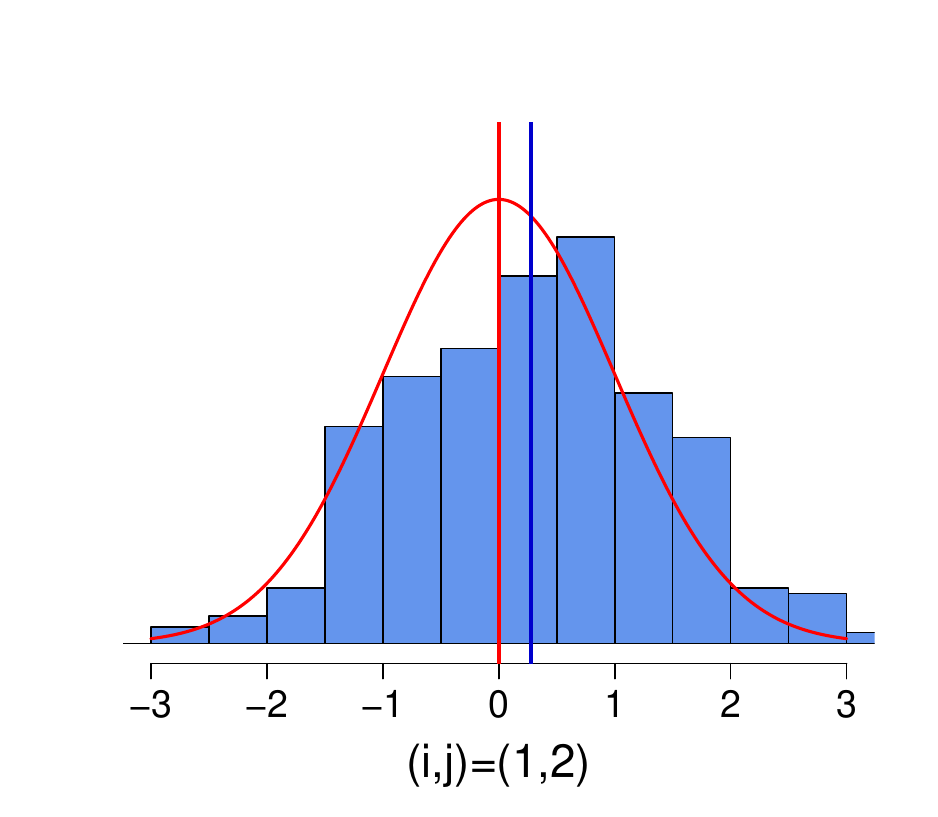}
 %          \caption*{\scriptsize $L_0{:}~ \widehat{\mb{T}}$}
    \end{minipage}
    \begin{minipage}{0.32\linewidth}
        \centering
        \includegraphics[width=\textwidth]{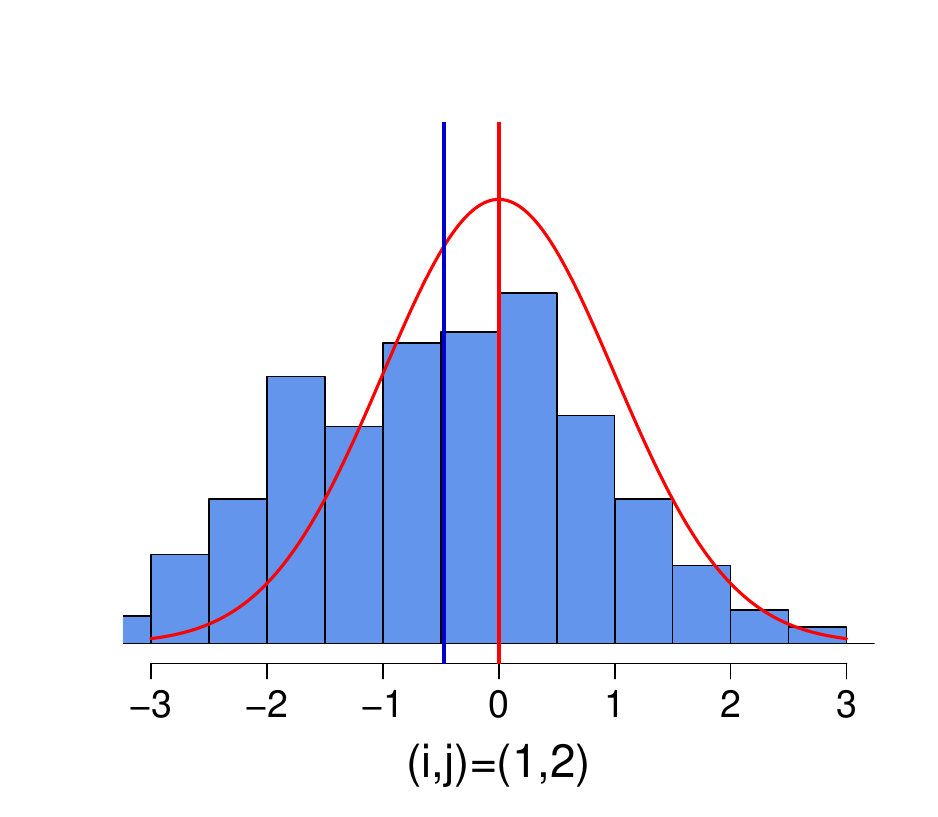}
%           \caption*{\scriptsize $L_1{:}~ \widehat{\mb{T}}$}
    \end{minipage}
 \end{minipage}
 \hfill
 \begin{minipage}{0.48\linewidth}
    \begin{minipage}{0.32\linewidth}
        \centering
        \includegraphics[width=\textwidth]{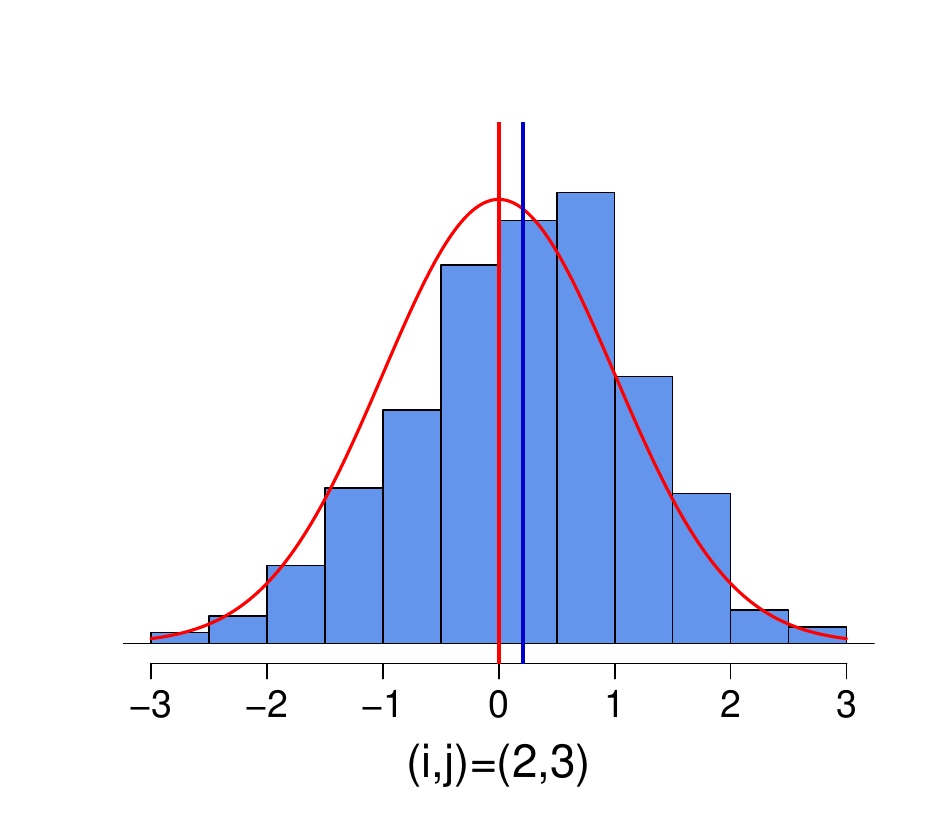}
 %                   \caption*{\scriptsize $L_0{:}~ \widehat{\mb{\Omega}}^{\text{US}}$}
    \end{minipage}
    \begin{minipage}{0.32\linewidth}
        \centering
        \includegraphics[width=\textwidth]{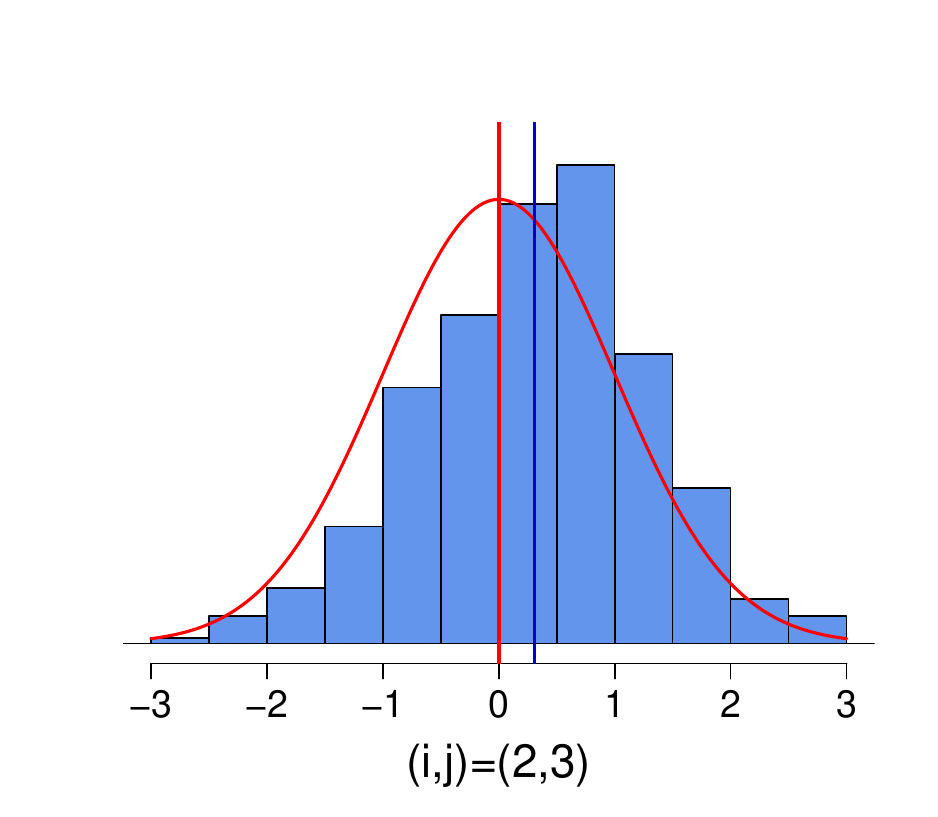}
 %          \caption*{\scriptsize $L_0{:}~ \widehat{\mb{T}}$}
    \end{minipage}
    \begin{minipage}{0.32\linewidth}
        \centering
        \includegraphics[width=\textwidth]{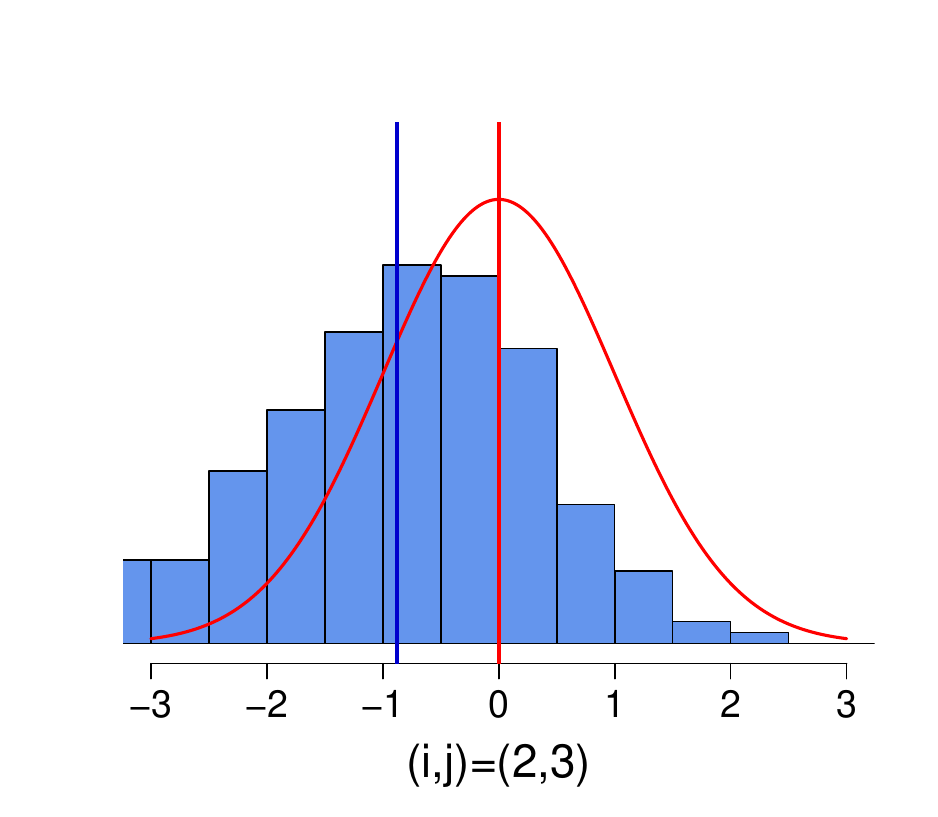}
%           \caption*{\scriptsize $L_1{:}~ \widehat{\mb{T}}$}
    \end{minipage}   
 \end{minipage}

  \caption*{\scriptsize $n=800, p=200$}
    \vspace{-0.43cm}
 \begin{minipage}{0.48\linewidth}
    \begin{minipage}{0.32\linewidth}
        \centering
        \includegraphics[width=\textwidth]{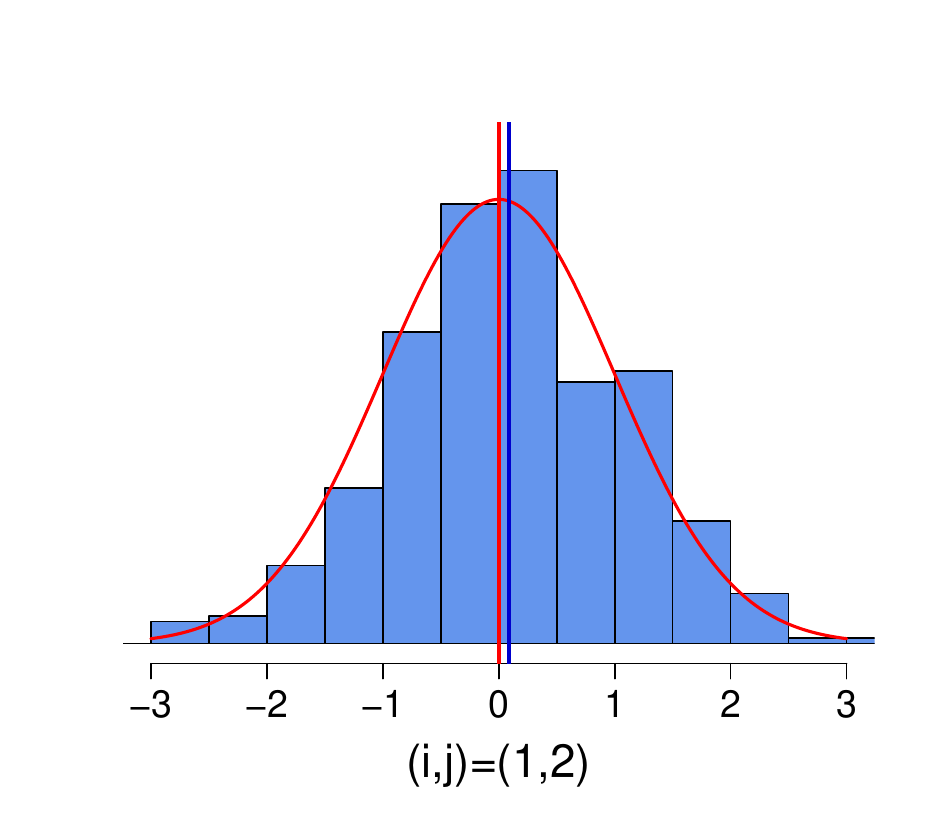}
%                            \caption*{\scriptsize $L_0{:}~ \widehat{\mb{\Omega}}^{\text{US}}$}
    \end{minipage}
    \begin{minipage}{0.32\linewidth}
        \centering
        \includegraphics[width=\textwidth]{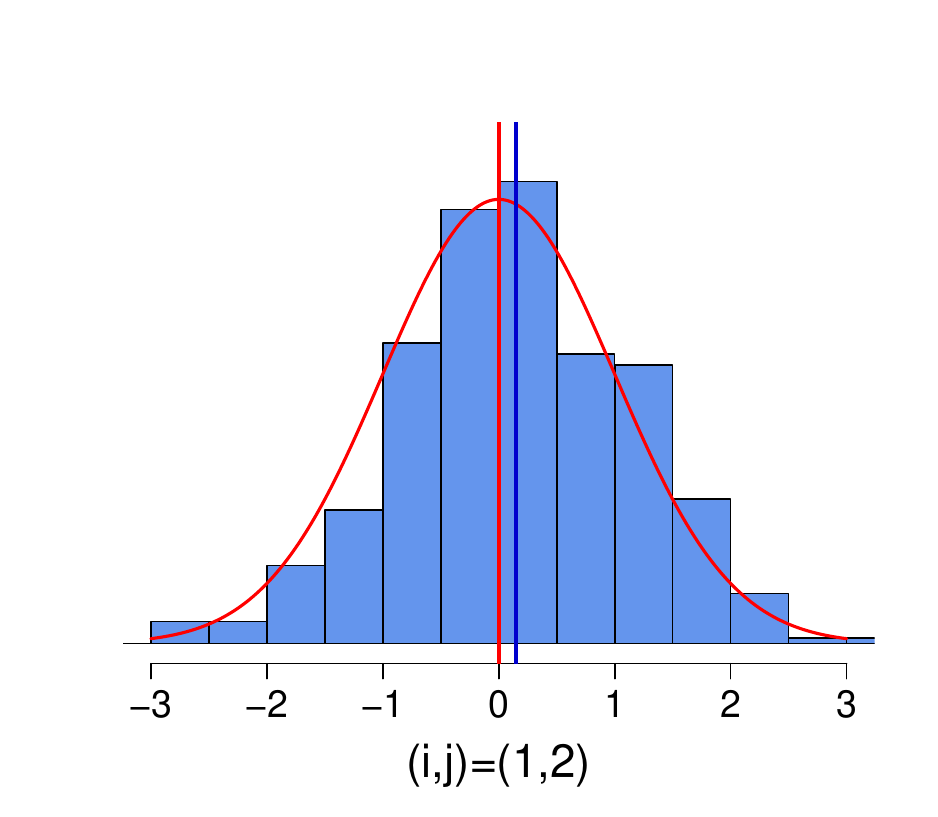}
%           \caption*{\scriptsize $L_0{:}~ \widehat{\mb{T}}$}
    \end{minipage}
    \begin{minipage}{0.32\linewidth}
        \centering
        \includegraphics[width=\textwidth]{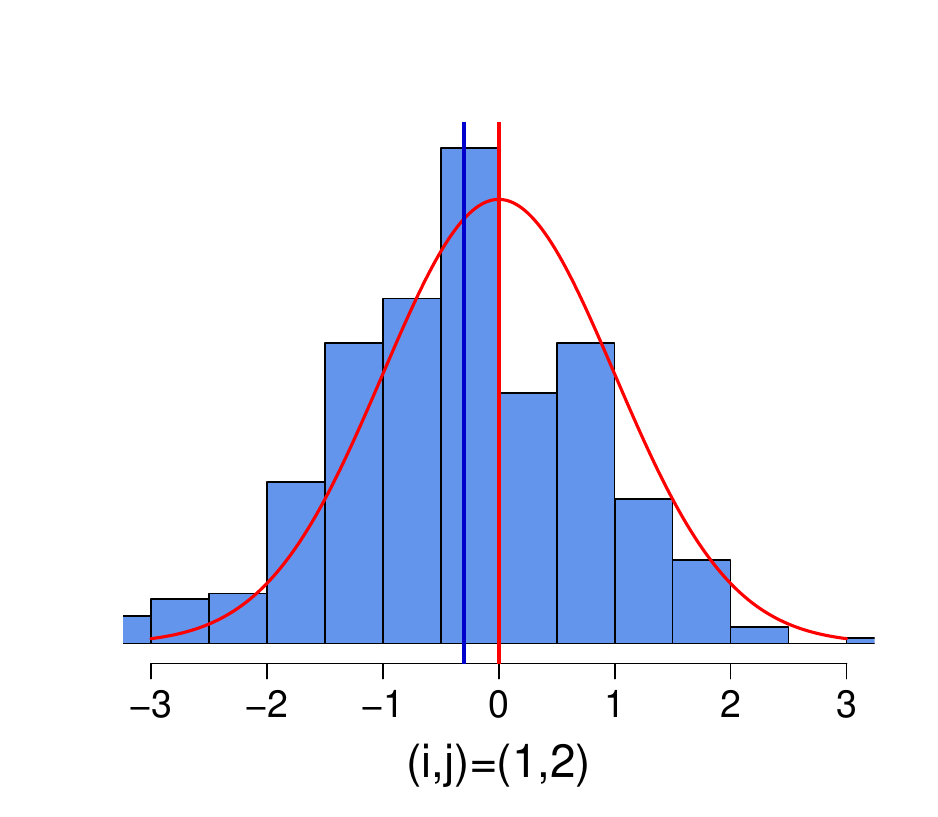}
%           \caption*{\scriptsize $L_1{:}~ \widehat{\mb{T}}$}
    \end{minipage}
 \end{minipage}
 \hfill
 \begin{minipage}{0.48\linewidth}
    \begin{minipage}{0.32\linewidth}
        \centering
        \includegraphics[width=\textwidth]{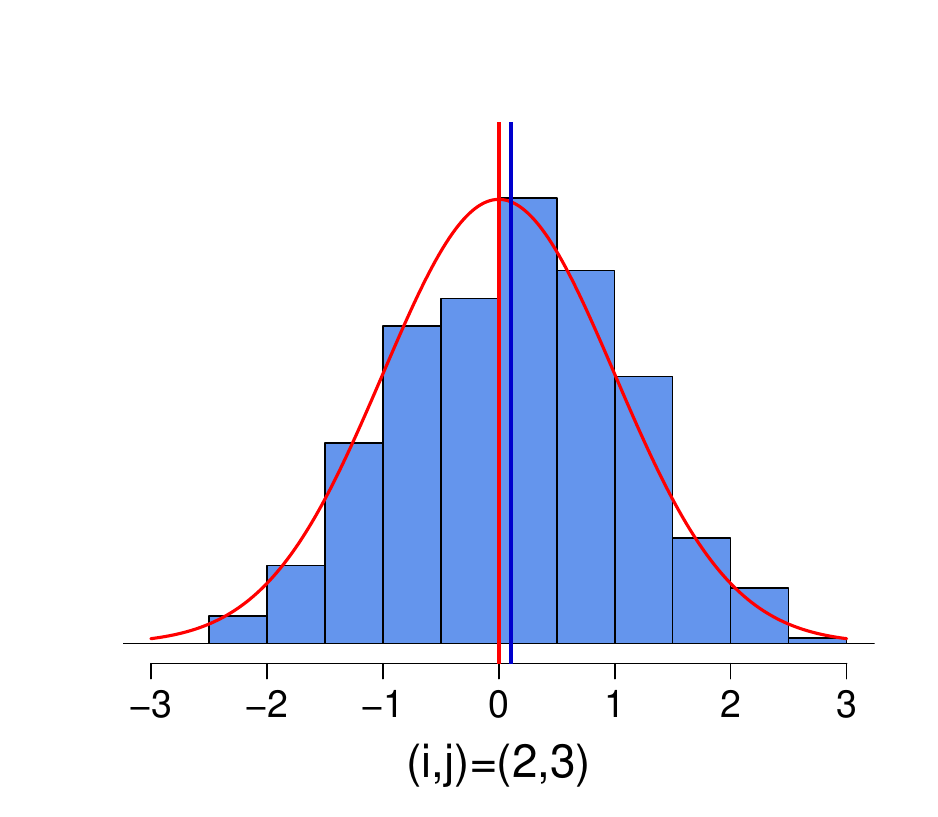}
%                            \caption*{\scriptsize $L_0{:}~ \widehat{\mb{\Omega}}^{\text{US}}$}
    \end{minipage}
    \begin{minipage}{0.32\linewidth}
        \centering
        \includegraphics[width=\textwidth]{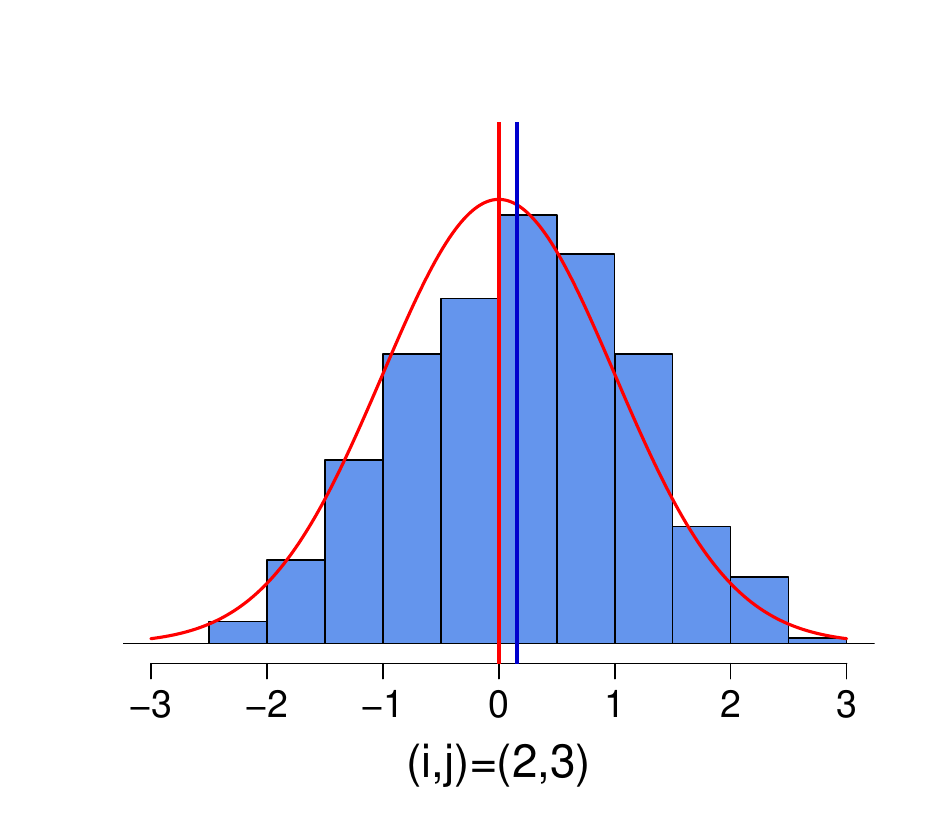}
%           \caption*{\scriptsize $L_0{:}~ \widehat{\mb{T}}$}
    \end{minipage}
    \begin{minipage}{0.32\linewidth}
        \centering
        \includegraphics[width=\textwidth]{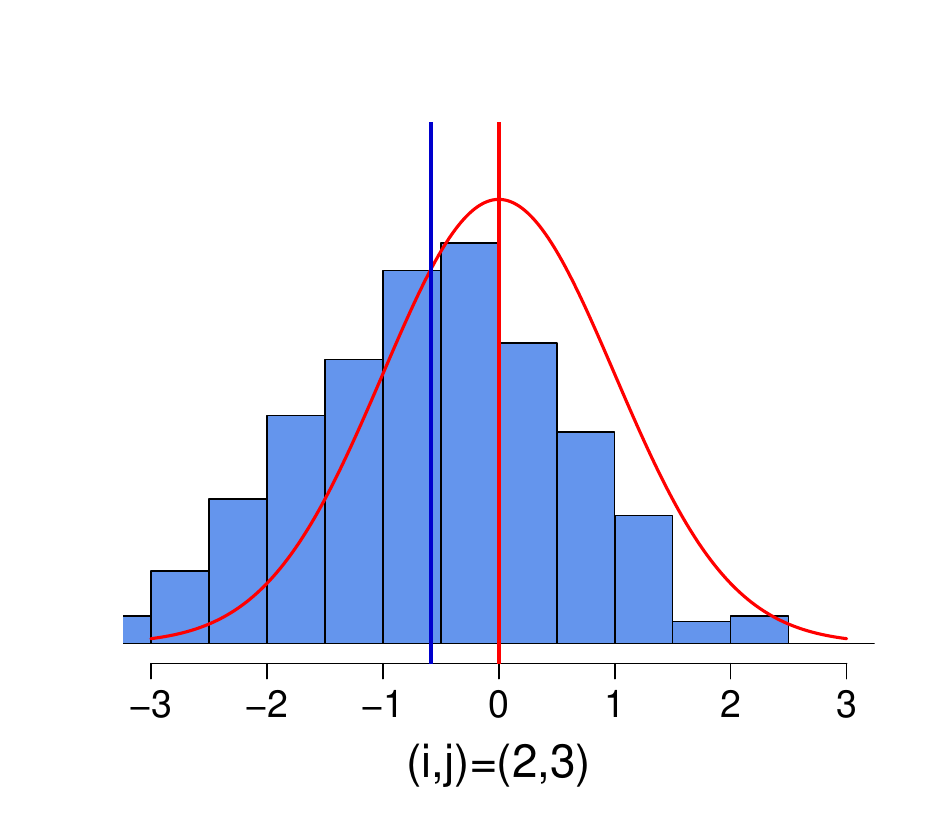}
%           \caption*{\scriptsize $L_1{:}~ \widehat{\mb{T}}$}
    \end{minipage}   
 \end{minipage}
 
   \caption*{\scriptsize $n=200, p=400$}
     \vspace{-0.43cm}
 \begin{minipage}{0.48\linewidth}
    \begin{minipage}{0.32\linewidth}
        \centering
        \includegraphics[width=\textwidth]{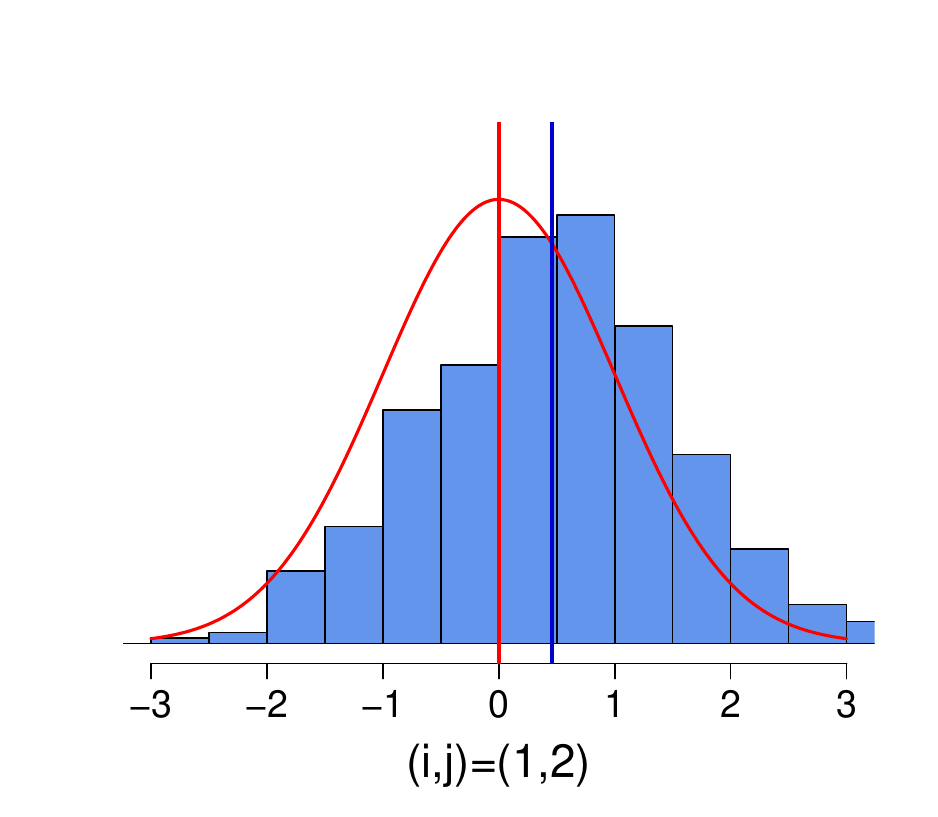}
%                            \caption*{\scriptsize $L_0{:}~ \widehat{\mb{\Omega}}^{\text{US}}$}
    \end{minipage}
    \begin{minipage}{0.32\linewidth}
        \centering
        \includegraphics[width=\textwidth]{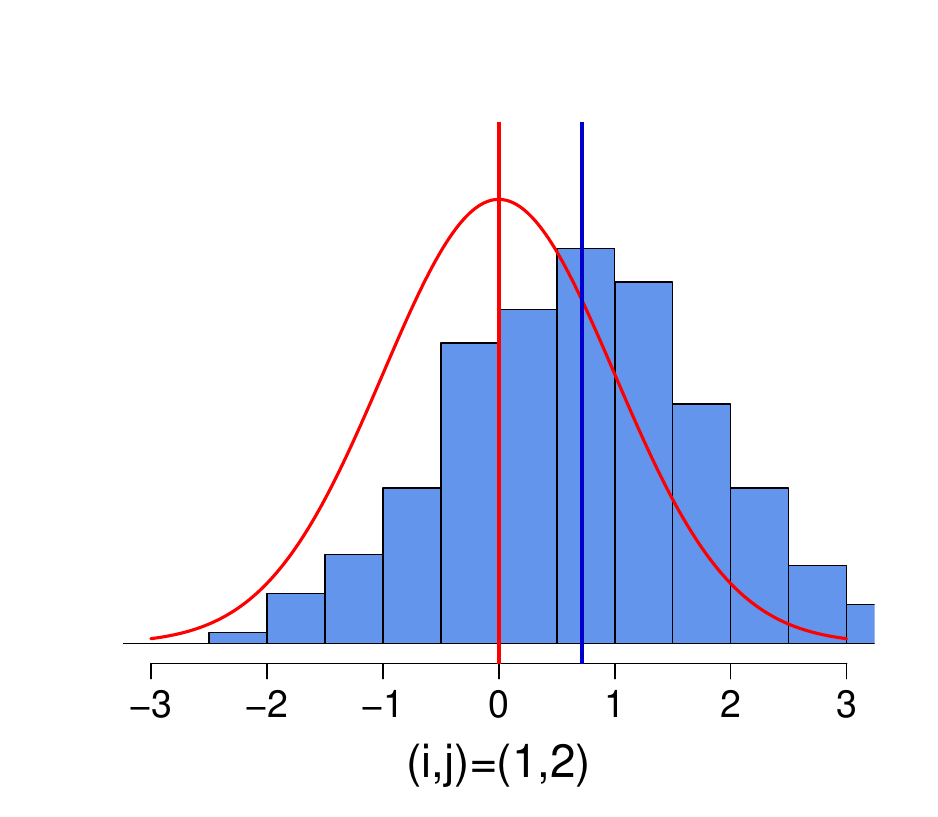}
%           \caption*{\scriptsize $L_0{:}~ \widehat{\mb{T}}$}
    \end{minipage}
    \begin{minipage}{0.32\linewidth}
        \centering
        \includegraphics[width=\textwidth]{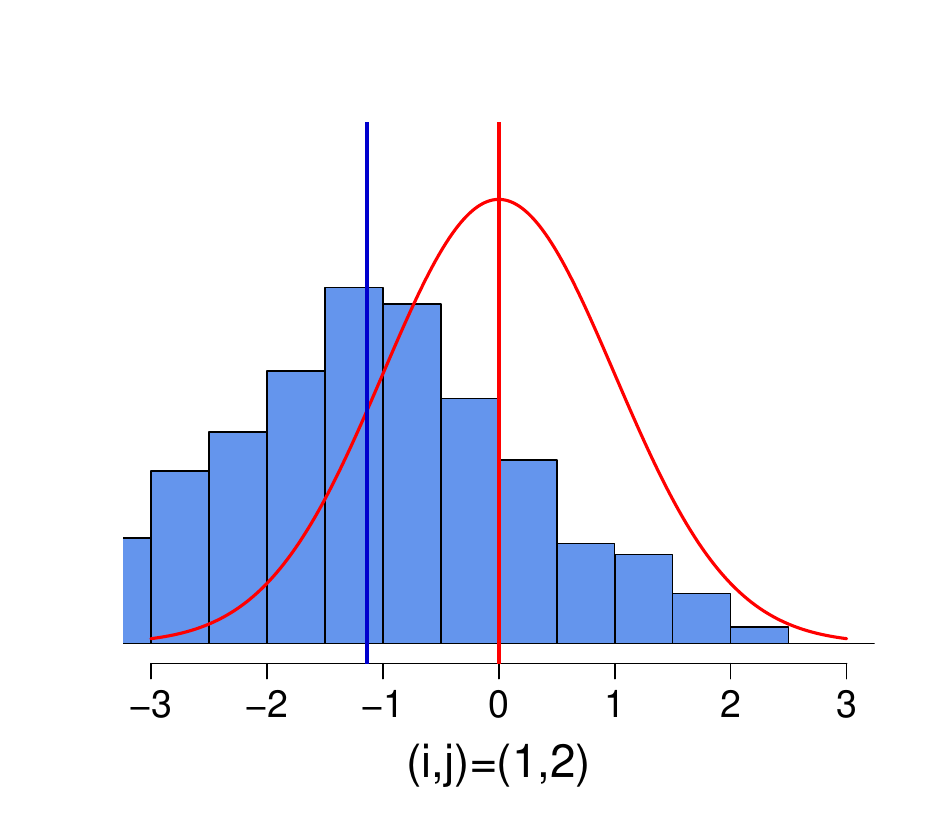}
%           \caption*{\scriptsize $L_1{:}~ \widehat{\mb{T}}$}
    \end{minipage}
 \end{minipage}
 \hfill
 \begin{minipage}{0.48\linewidth}
    \begin{minipage}{0.32\linewidth}
        \centering
        \includegraphics[width=\textwidth]{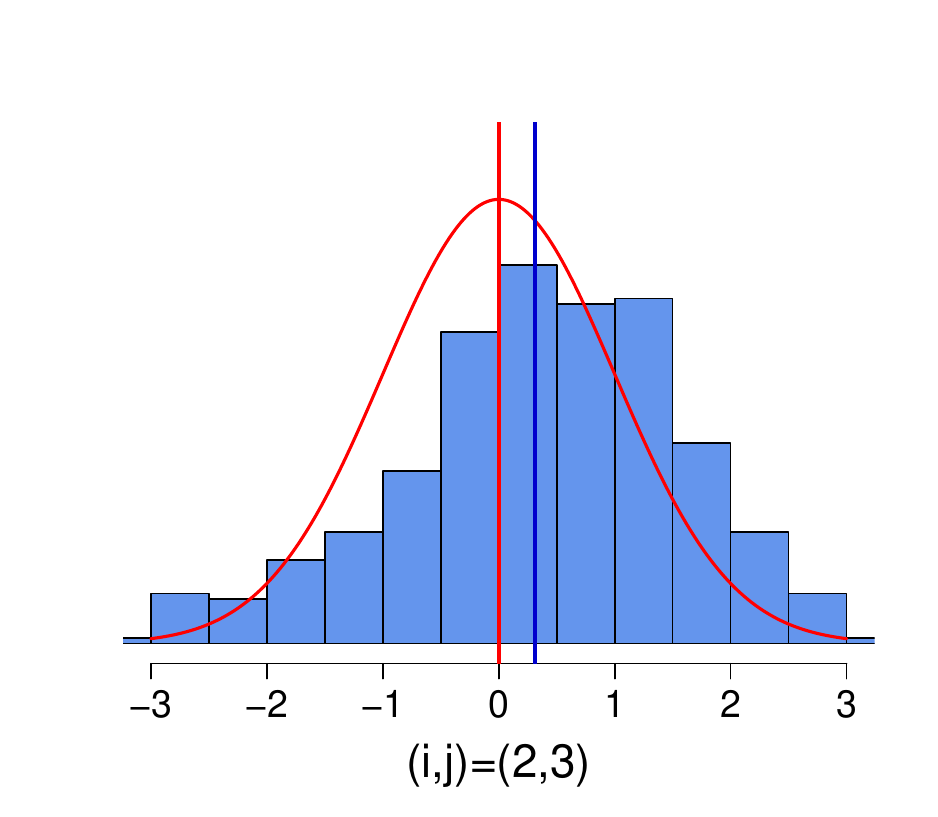}
%                            \caption*{\scriptsize $L_0{:}~ \widehat{\mb{\Omega}}^{\text{US}}$}
    \end{minipage}
    \begin{minipage}{0.32\linewidth}
        \centering
        \includegraphics[width=\textwidth]{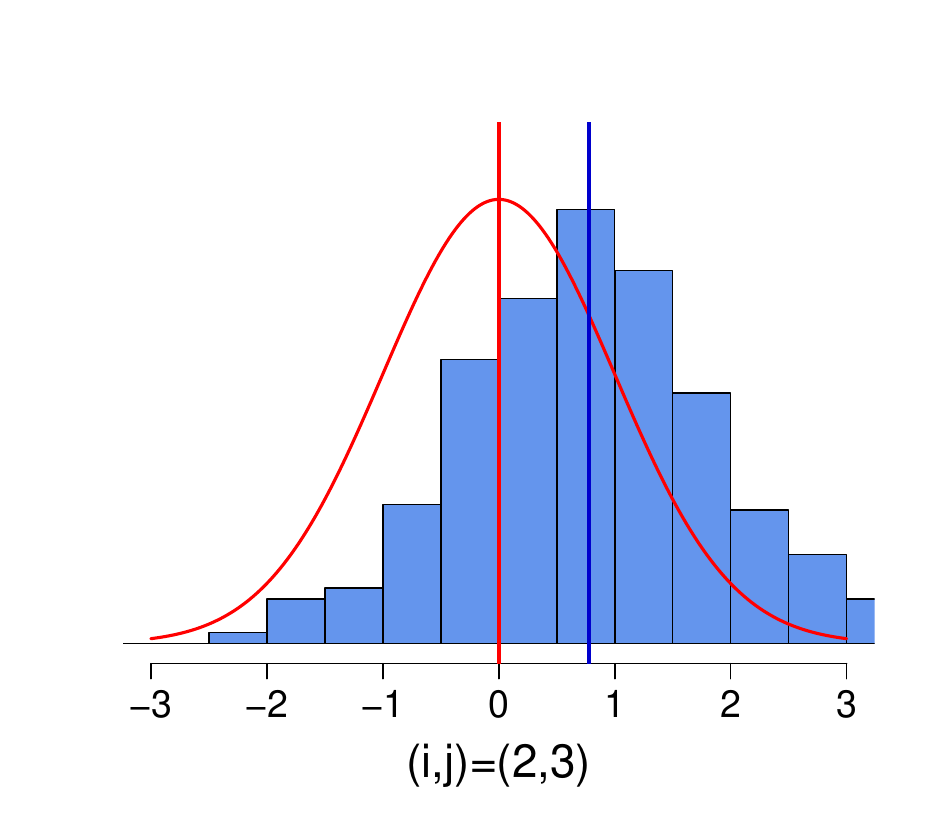}
%           \caption*{\scriptsize $L_0{:}~ \widehat{\mb{T}}$}
    \end{minipage}
    \begin{minipage}{0.32\linewidth}
        \centering
        \includegraphics[width=\textwidth]{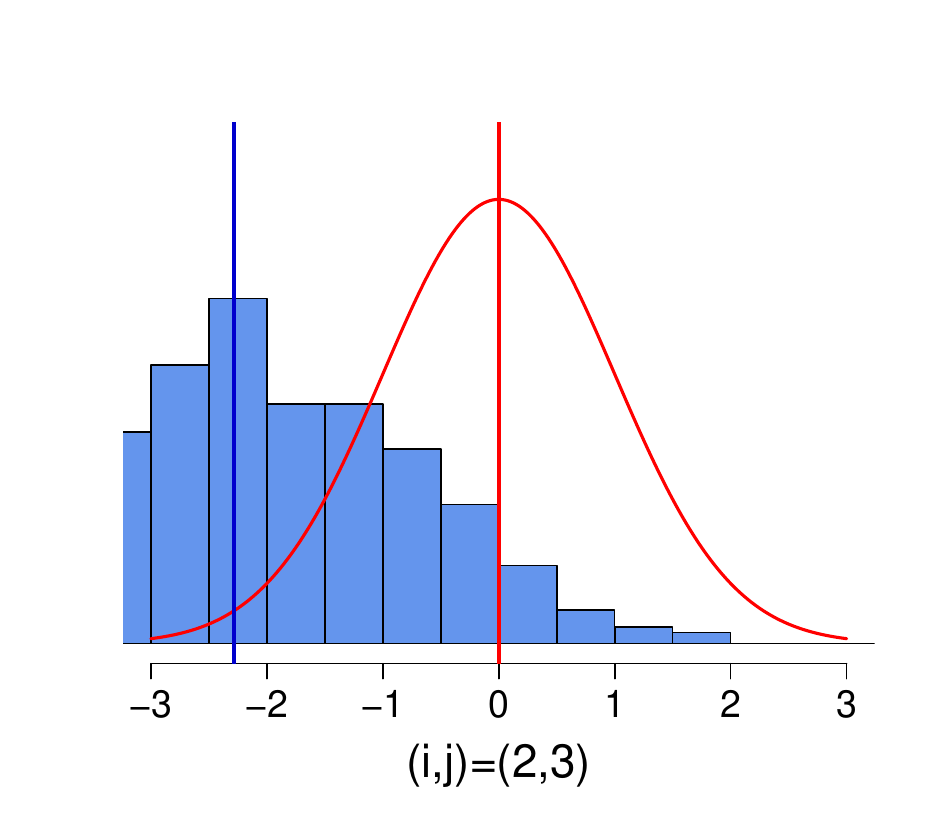}
%           \caption*{\scriptsize $L_1{:}~ \widehat{\mb{T}}$}
    \end{minipage}   
 \end{minipage}

  \caption*{\scriptsize $n=400, p=400$}
    \vspace{-0.43cm}
 \begin{minipage}{0.48\linewidth}
    \begin{minipage}{0.32\linewidth}
        \centering
        \includegraphics[width=\textwidth]{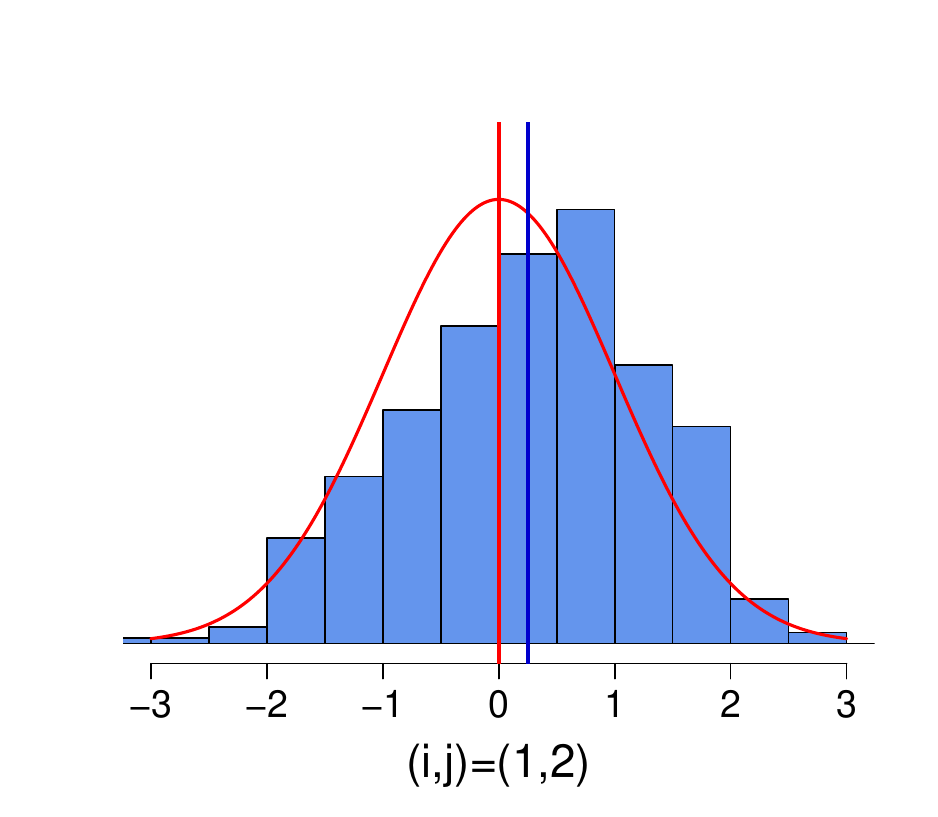}
%                            \caption*{\scriptsize $L_0{:}~ \widehat{\mb{\Omega}}^{\text{US}}$}
    \end{minipage}
    \begin{minipage}{0.32\linewidth}
        \centering
        \includegraphics[width=\textwidth]{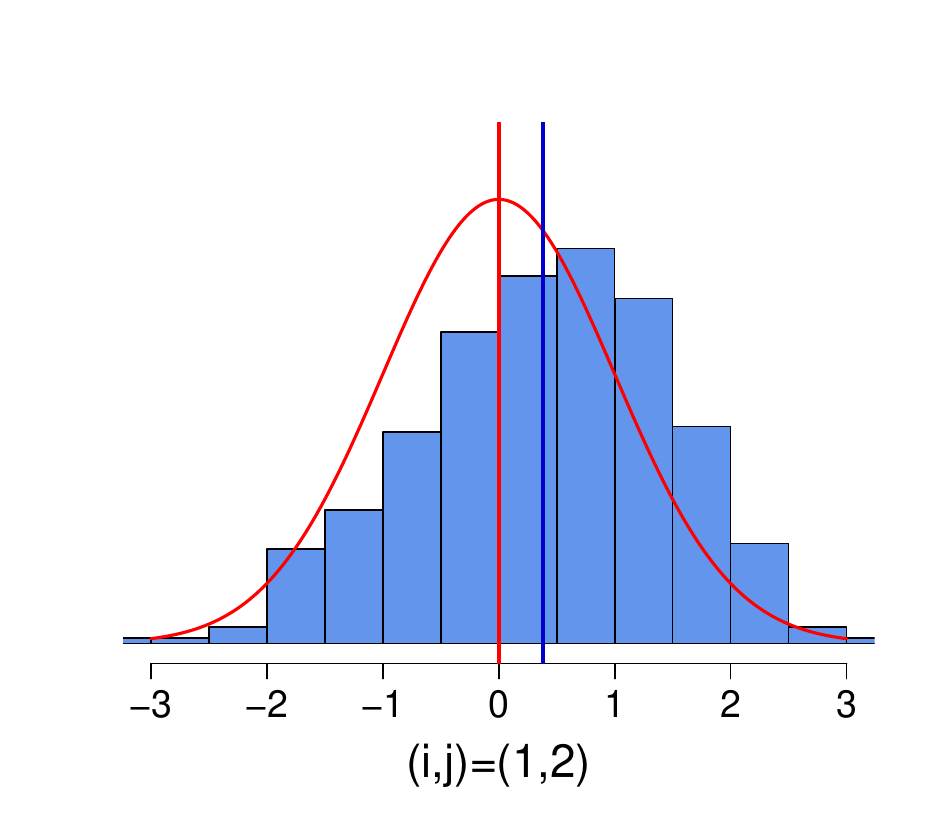}
%           \caption*{\scriptsize $L_0{:}~ \widehat{\mb{T}}$}
    \end{minipage}
    \begin{minipage}{0.32\linewidth}
        \centering
        \includegraphics[width=\textwidth]{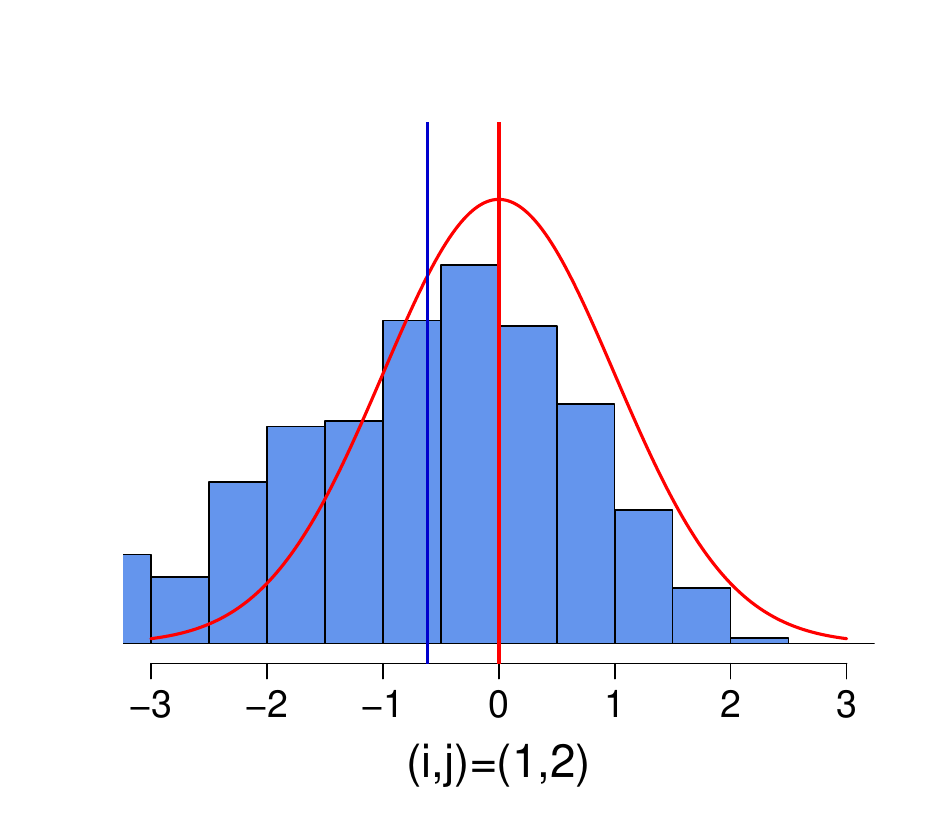}
%           \caption*{\scriptsize $L_1{:}~ \widehat{\mb{T}}$}
    \end{minipage}
 \end{minipage}
 \hfill
 \begin{minipage}{0.48\linewidth}
    \begin{minipage}{0.32\linewidth}
        \centering
        \includegraphics[width=\textwidth]{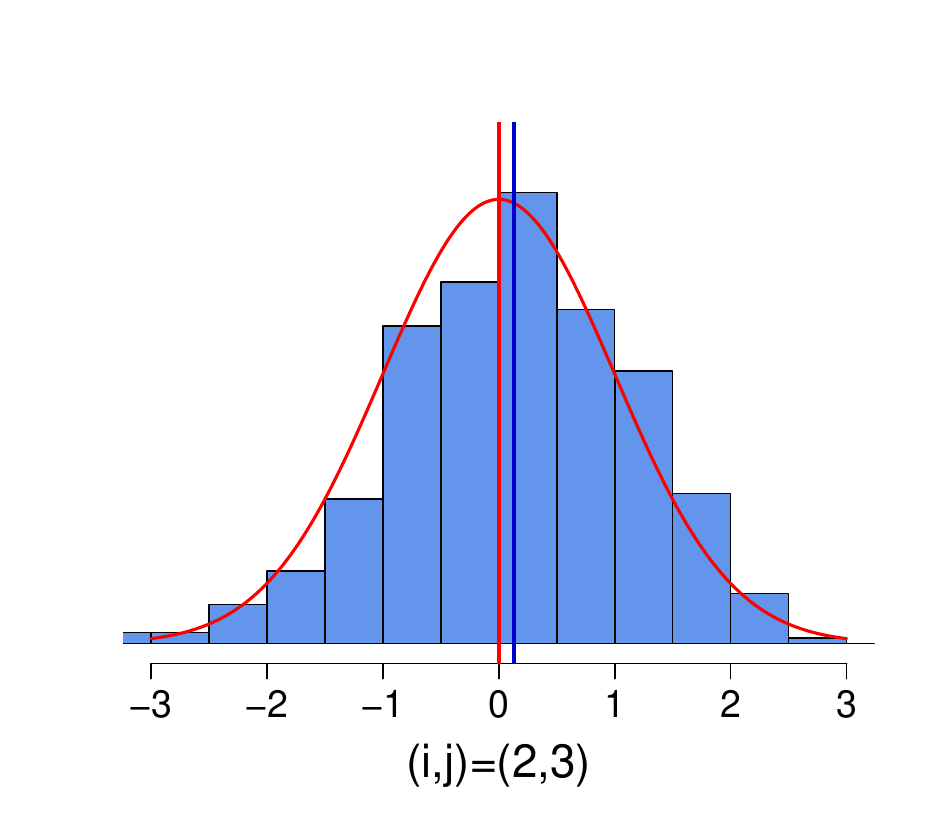}
%                            \caption*{\scriptsize $L_0{:}~ \widehat{\mb{\Omega}}^{\text{US}}$}
    \end{minipage}
    \begin{minipage}{0.32\linewidth}
        \centering
        \includegraphics[width=\textwidth]{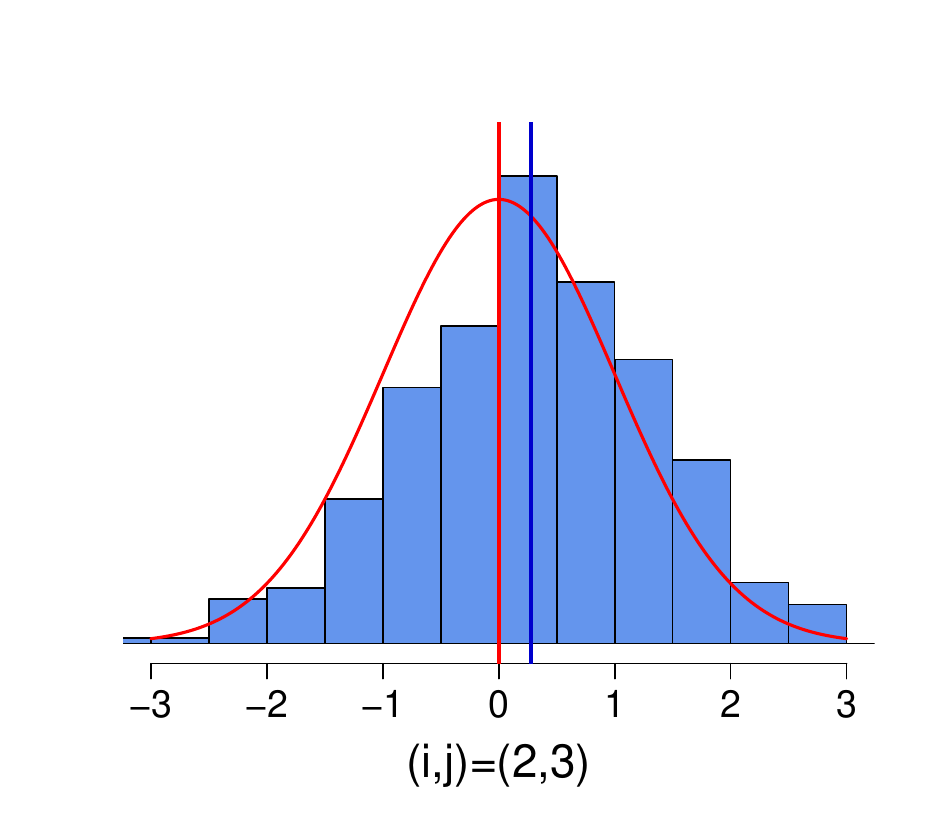}
%           \caption*{\scriptsize $L_0{:}~ \widehat{\mb{T}}$}
    \end{minipage}
    \begin{minipage}{0.32\linewidth}
        \centering
        \includegraphics[width=\textwidth]{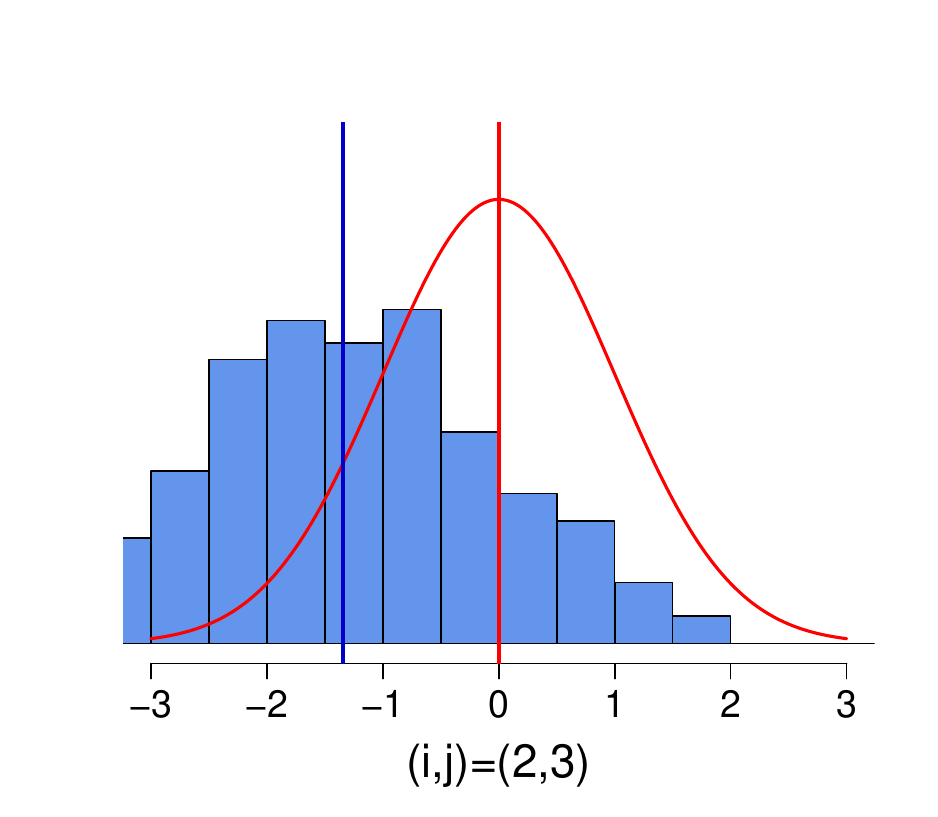}
%           \caption*{\scriptsize $L_1{:}~ \widehat{\mb{T}}$}
    \end{minipage}   
 \end{minipage}

  \caption*{\scriptsize $n=800, p=400$}
    \vspace{-0.43cm}
 \begin{minipage}{0.48\linewidth}
    \begin{minipage}{0.32\linewidth}
        \centering
        \includegraphics[width=\textwidth]{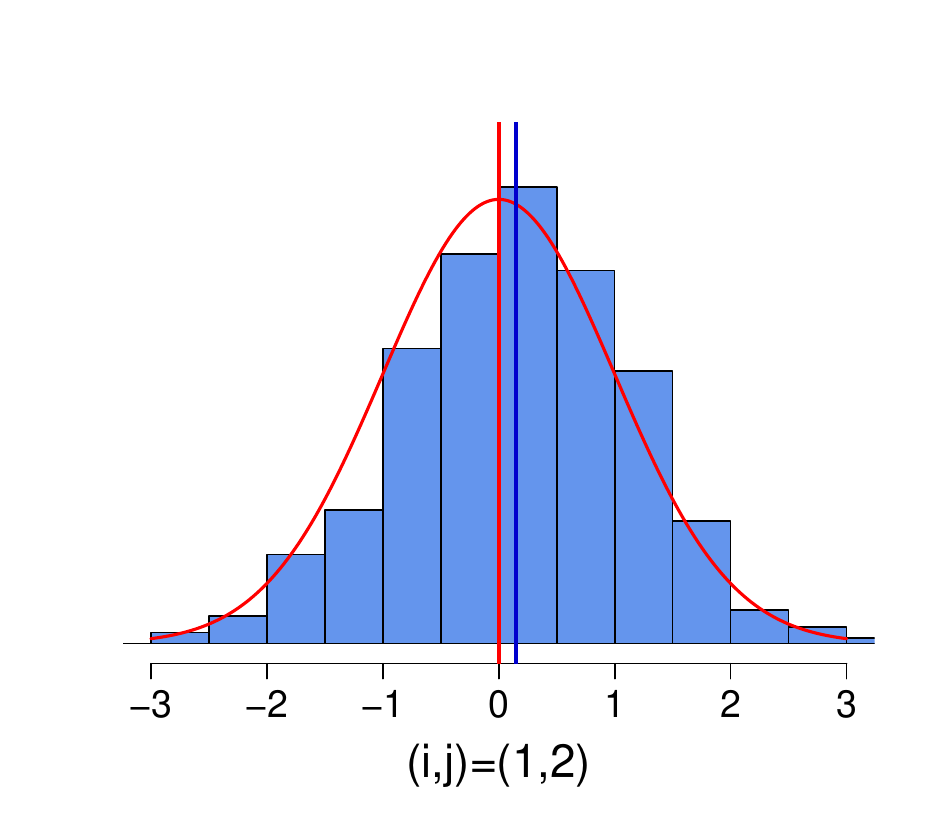}
                           \caption*{\scriptsize $L_0{:}~ \widehat{\mb{\Omega}}^{\text{US}}$ }
    \end{minipage}
    \begin{minipage}{0.32\linewidth}
        \centering
        \includegraphics[width=\textwidth]{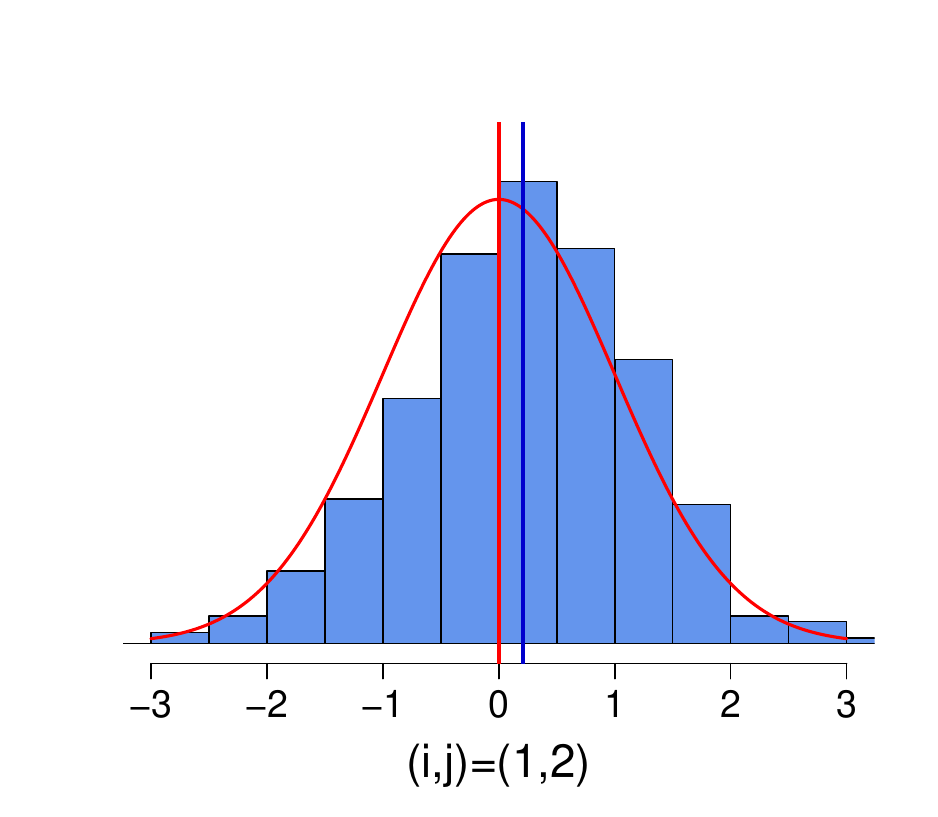}
           \caption*{\scriptsize $L_0{:}~ \widehat{\mb{T}}$}
    \end{minipage}
    \begin{minipage}{0.32\linewidth}
        \centering
        \includegraphics[width=\textwidth]{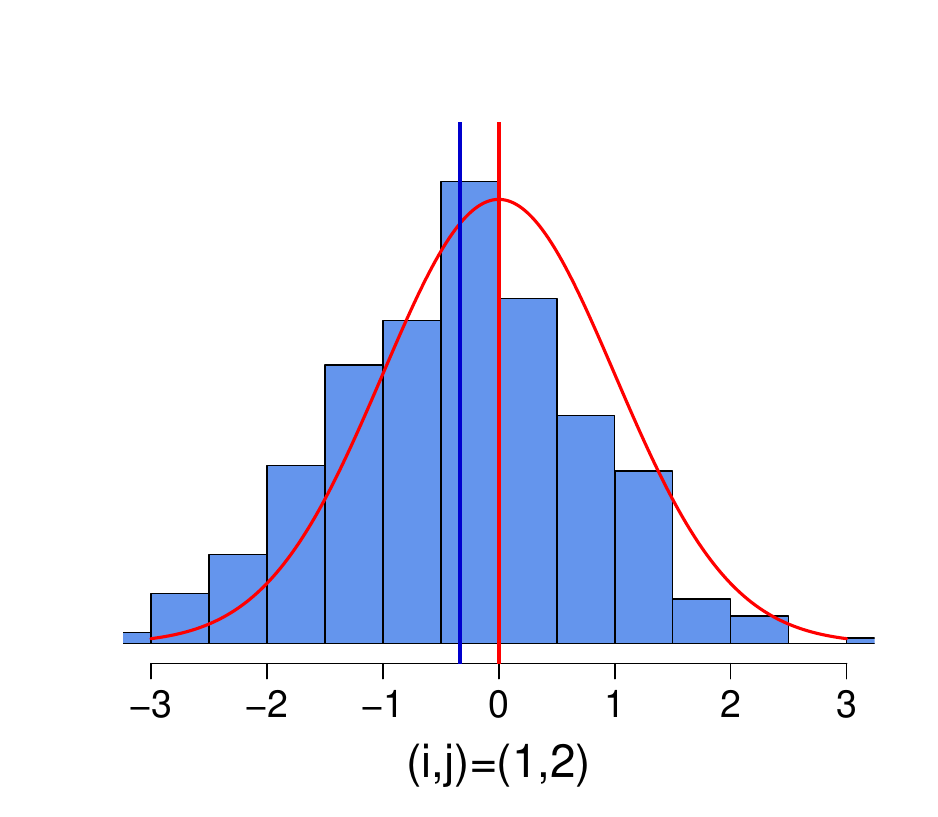}
           \caption*{\scriptsize $L_1{:}~ \widehat{\mb{T}}$}
    \end{minipage}
     \caption*{(a) $(i,j)=(1,2)$}
 \end{minipage}
 \hfill
 \begin{minipage}{0.48\linewidth}
    \begin{minipage}{0.32\linewidth}
        \centering
        \includegraphics[width=\textwidth]{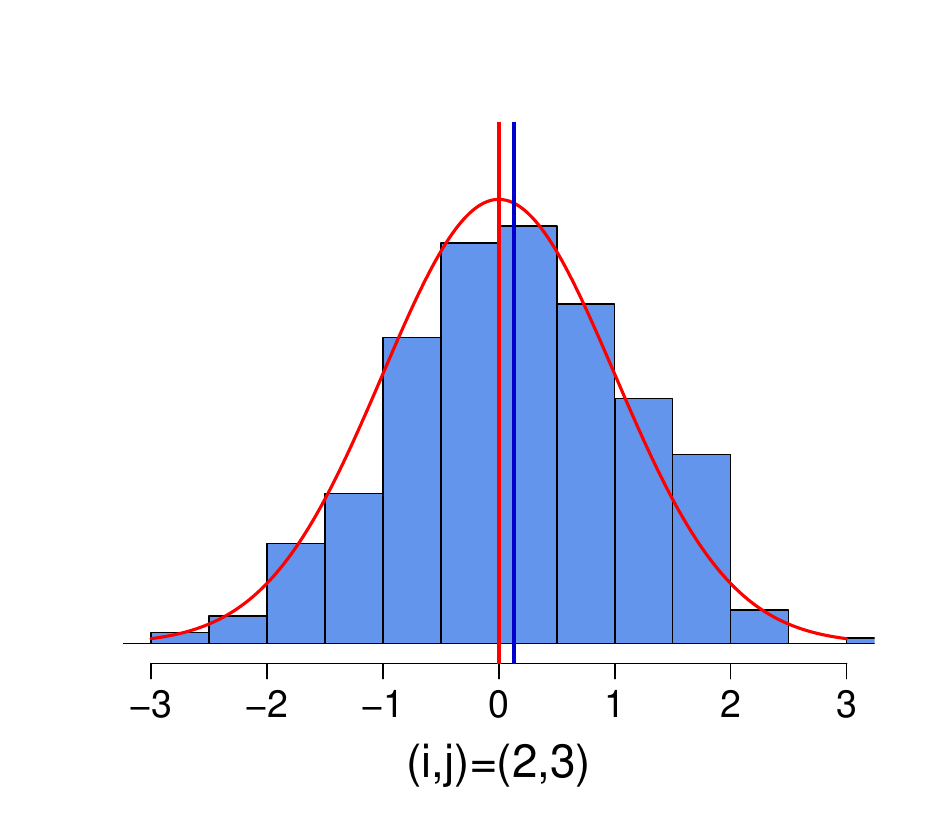}
                            \caption*{\scriptsize $L_0{:}~ \widehat{\mb{\Omega}}^{\text{US}}$}
    \end{minipage}
    \begin{minipage}{0.32\linewidth}
        \centering
        \includegraphics[width=\textwidth]{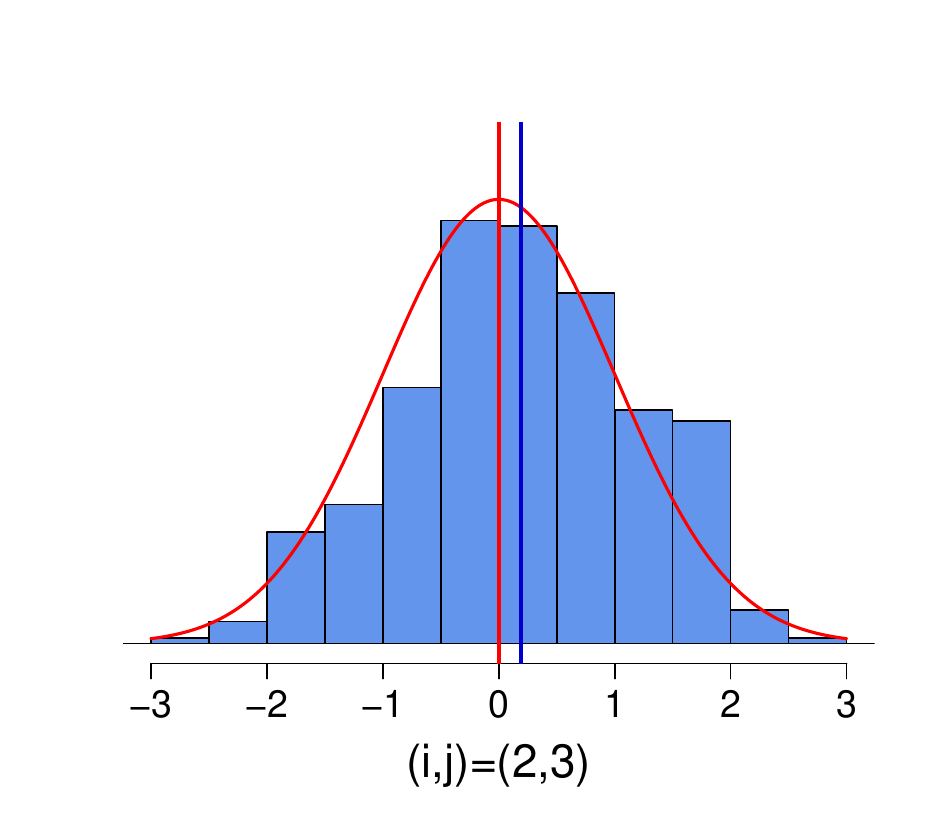}
           \caption*{\scriptsize $L_0{:}~ \widehat{\mb{T}}$}
    \end{minipage}
    \begin{minipage}{0.32\linewidth}
        \centering
        \includegraphics[width=\textwidth]{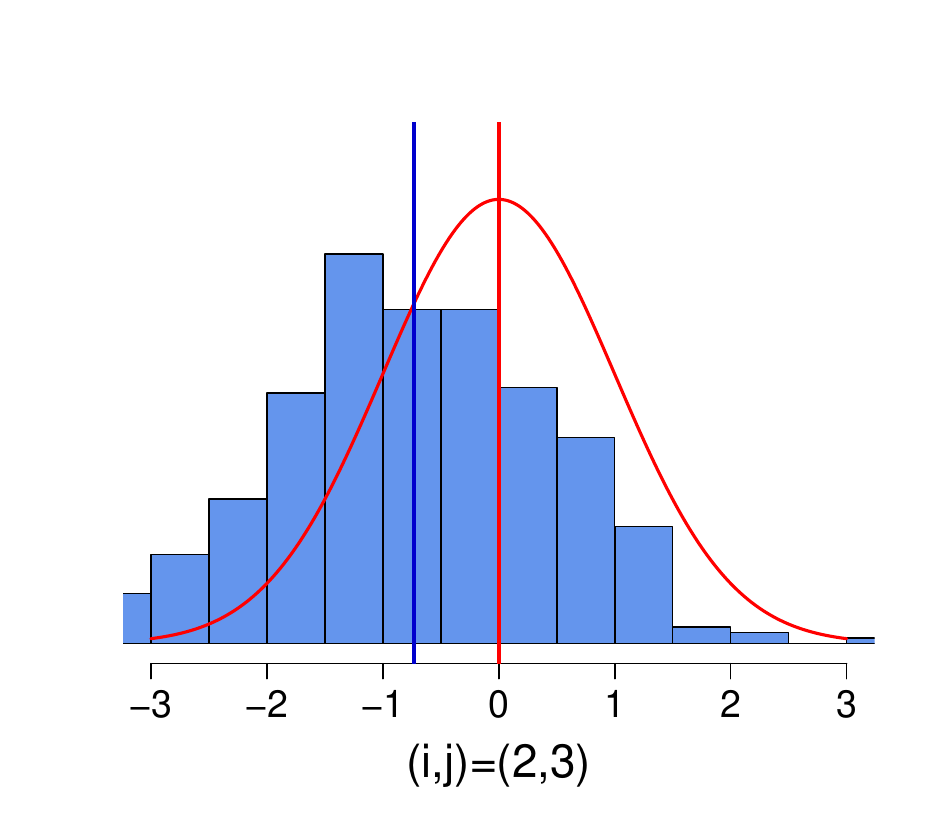}
           \caption*{\scriptsize $L_1{:}~ \widehat{\mb{T}}$}
    \end{minipage}   
         \caption*{(b) $(i,j)=(2,3)$}
 \end{minipage}

     \caption{Histograms of $\big(\sqrt{n}(\widehat{\mb{\Omega}}_{ij}^{(m)}-\mb{\Omega}_{ij})/\widehat{\sigma}_{\mb{\Omega}_{ij}}^{(m)}\big)_{m=1}^{400}$ under Gaussian band graph settings.
     The results for more matrix entries and other (sub-)Gaussian graph settings are given in 
     Figures \ref{fig: normalplot Gaussian band}--\ref{fig: normalplot sub-Gaussian cluster}
in Supplementary Materials.}
\label{fig: normalplot Gaussian band selected}
\end{figure}

 To visualize the performance of the three estimators on asymptotic normality, we present histograms of their \(Z\)-score values \(\big(\sqrt{n}(\widehat{\mathbf{\Omega}}_{ij}^{(m)} - \mathbf{\Omega}_{ij}) / \widehat{\sigma}_{\mathbf{\Omega}_{ij}}^{(m)}\big)_{m=1}^{M}\) from \(M = 400\) replications in Figure~\ref{fig: normalplot Gaussian band selected} and Figures~\ref{fig: normalplot Gaussian band}--\ref{fig: normalplot sub-Gaussian cluster} in the Supplementary Materials, superimposed with the density curve of the standard Gaussian distribution.

All three methods show \(Z\)-score histograms approaching the standard Gaussian curve as \(n\) increases.  The Nodewise Loreg estimators \(\widehat{\mathbf{\Omega}}^{\text{US}}\) and \(\widehat{\mathbf{T}}\) exhibit similar performance, again suggesting that debiasing \(\widehat{\mathbf{\Omega}}^{\text{US}}\) may not be necessary. Moreover, they outperform the Nodewise Lasso estimator \(\widehat{\mathbf{T}}\) in demonstrating asymptotic normality.
Particularly, the two Nodewise Loreg estimators have \(Z\)-score histograms well approximated by the standard Gaussian curve in most examples when \(n \le p\), though the theorems for their asymptotic normality (Theorems~\ref{Corollary: AN} and \ref{thm: normality 2}) require the strong assumption \(\max_{j\in[p]}s_j(\log p) / \sqrt{n} = o(1)\)
as in Nodewise Lasso’s \citep[][Theorem 1]{Jank17}. In contrast, the Nodewise Lasso estimator \(\widehat{\mathbf{T}}\)  often fails to fit the curve well, even when \(n = 800\), especially under the cluster graph settings.

\subsection{Timing performance} 

We execute the Nodewise Loreg, Nodewise Lasso, CLIME, and GLasso using our institution's High Performance Computing cluster. 
Each simulation replication is allocated with 2GB memory on an Intel Xeon Platinum 8268 CPU core operating at 2.90GHz. 
We only compute the runtime of the respective estimator $\widehat{\mb{\Omega}}^{\text{S}}$ for Nodewise Loreg and Nodewise Lasso, 
without considering the optional post-thresholding using multiple testing,
which varies in time depending on the selected FDR control method.
Table~\ref{tab: runtime} summarizes the average and standard deviation of the runtime  over the 100 replications under each simulation setting for each method.
Notably,
our Nodewise Loreg outperforms the other methods in 13 of the 32 settings,
achieving significantly shorter runtime than Nodewise Lasso
with a reduction of up to 45\% in average runtime for all settings, 
except the two hub graph settings with $(p,n)=(400,400)$.
CLIME exhibits the longest runtime in every setting,
with its average runtime more than 100 times that of Nodewise Loreg in 22 settings.
While GLasso records the shortest runtime in 19 settings, 16 of which have $n=400$,
it requires up to 54\% longer  average runtime in the corresponding 16 settings with $n=200$,
where  GLasso is slower than Nodewise Loreg in 13 settings.
This counterintuitive phenomenon of  increased runtime for smaller $n$ is also observed 
for CLIME when $p=200$.
This occurs because GLasso and CLIME only require the sample covariance matrix as input data, without the need for the sample size $n$. 
Such significant variability in runtime across settings with the same $p$ but different $n$
highlights the unstable timing performance of GLasso.
As expected, the runtime increases as either $n$ or $p$ increases for Nodewise Loreg and Nodewise Lasso.
Moreover, the column-by-column estimating nature of Nodewise Loreg, Nodewise Lasso, and CLIME  allows for acceleration using
parallel computing on multiple CPU cores, an advantage that GLasso lacks.

 \begin{table}[h!]
 \renewcommand\arraystretch{1.3}
 \begin{center}
\resizebox{\textwidth}{!}{
\begin{tabular}{cccccccc}
\toprule
Graph&     & Gaussian: $p=200$& Gaussian: $p=400$& Sub-Gaussian: $p=200$&Sub-Gaussian: $p=400$  \\
Setting&    Method & $n=200/n=400$& $n=200/n=400$ &$n=200/n=400$& $n=200/n=400$  \\
\midrule

\multirow{4}{*}{Band}&$L_0{:}~ \widehat{\mb{\Omega}}^{\text{S}}$
&{\bf 5.335(0.219)}/9.111(0.246) &{\bf 22.434(0.822)}/41.196(2.315) &{\bf 5.297(0.197)}/9.209(0.270) &{\bf 22.472(0.785)}/40.648(2.203) \\
\multirow{4}{*}{}&$L_1{:}~ \widehat{\mb{\Omega}}^{\text{S}}$
&9.513(0.630)/12.27(0.323) &38.927(2.095)/55.500(2.060) &9.719(0.763)/12.36(0.457) &39.496(2.198)/54.575(2.175) \\
\multirow{4}{*}{}& CLIME
&578.0(3.380)/182.9(1.213) &2687.4(53.87)/4997.6(184.8) &577.7(3.325)/182.6(1.423) &2641.8(75.47)/4975.7(174.9) \\
\multirow{4}{*}{}& GLasso
&5.791(0.226)/{\bf 4.519(0.046)} &56.328(2.206)/{\bf 36.860(1.184)} &5.753(0.220)/{\bf 4.529(0.062)} &57.388(2.136)/{\bf 37.239(1.407)} \\

\midrule
\multirow{4}{*}{Random}&$L_0{:}~ \widehat{\mb{\Omega}}^{\text{S}}$
&{\bf 5.241(0.230)}/8.975(0.328) &{\bf 22.351(0.865)}/40.906(1.739) &{\bf 5.113(0.197)}/8.775(0.289) &{\bf 22.107(0.839)}/40.901(1.854) \\
\multirow{4}{*}{}&$L_1{:}~ \widehat{\mb{\Omega}}^{\text{S}}$
&9.007(0.437)/11.72(0.614) &33.733(0.597)/49.489(2.340) &9.208(0.475)/11.85(0.622) &33.752(0.718)/49.643(2.800) \\
\multirow{4}{*}{}& CLIME
&571.9(19.42)/180.5(5.070) &2618.6(69.71)/5080.0(53.11) &572.5(17.44)/180.6(4.681) &2617.3(81.80)/5040.2(85.89) \\
\multirow{4}{*}{}& GLasso
&5.543(0.366)/{\bf 3.812(0.215)} &48.725(3.261)/{\bf 34.274(1.565)} &5.663(0.352)/{\bf 3.816(0.192)} &48.929(3.222)/{\bf 34.207(1.512)} \\

\midrule
\multirow{4}{*}{Hub}&$L_0{:}~ \widehat{\mb{\Omega}}^{\text{S}}$
& 5.932(0.178)/10.73(0.261) &{\bf 24.319(0.616)}/44.034(2.488) &{\bf 5.826(0.169)}/10.53(0.282) &{\bf 23.947(0.652)}/43.753(2.029) \\
\multirow{4}{*}{}&$L_1{:}~ \widehat{\mb{\Omega}}^{\text{S}}$
&8.374(0.392)/11.12(0.503) &28.248(2.020)/40.368(3.410) &8.580(0.421)/11.10(0.528) &28.416(2.025)/40.291(3.620) \\
\multirow{4}{*}{}& CLIME
&560.2(14.25)/196.3(7.235) &2665.2(67.20)/4806.5(37.98) &559.4(14.29)/196.0(6.804) &2660.3(71.48)/4793.6(39.04) \\
\multirow{4}{*}{}& GLasso
&{\bf 5.854(0.299)}/{\bf 4.610(0.195)} &49.782(2.863)/{\bf 36.676(1.709)} &6.008(0.294)/{\bf 4.574(0.204)} &49.762(2.706)/{\bf 36.840(1.863)} \\

\midrule
\multirow{4}{*}{Cluster}&$L_0{:}~ \widehat{\mb{\Omega}}^{\text{S}}$
&5.148(0.169)/8.992(0.210) &{\bf 21.933(0.958)}/38.489(2.411) &5.061(0.164)/8.849(0.207) &{\bf 21.639(0.936)}/38.647(2.232) \\
\multirow{4}{*}{}&$L_1{:}~ \widehat{\mb{\Omega}}^{\text{S}}$
&8.910(0.212)/11.73(0.679) &33.537(2.272)/51.223(9.438) &9.030(0.247)/11.86(0.729) &33.515(2.394)/51.020(9.905) \\
\multirow{4}{*}{}& CLIME
&566.1(5.136)/181.3(5.235) &2453.5(23.43)/4719.1(211.9) &556.5(4.938)/181.0(5.203) &2451.9(22.73)/4543.9(262.2) \\
\multirow{4}{*}{}& GLasso
&{\bf 4.538(0.249)}/{\bf 3.423(0.129)} &34.968(1.797)/{\bf 24.114(1.060)} &{\bf 4.654(0.244)}/{\bf 3.421(0.108)} &34.765(1.856)/{\bf 24.347(1.259)} \\
\bottomrule
\end{tabular}}
\caption{Average (standard deviation) of runtime in seconds over 100 simulation replications.}
\label{tab: runtime}
\end{center}
 \end{table}

\section{Analysis of MDA133 breast cancer gene expression}\label{sec: real data}

	We compare the performance of 
Nodewise Loreg, Nodewise Lasso, CLIME, and GLasso	
on precision matrix estimation 
using the MDA133 breast cancer gene expression dataset \citep{hess2006pharmacogenomic}
obtained from MD Anderson Cancer Center (\url{https://bioinformatics.mdanderson.org/public-datasets}). This dataset contains
the dChip normalized model based expression index (MBEI) values at 22,283 Affymetrix gene probesets
for 133 breast cancer patients treated with neoadjuvant chemotherapy,
including 34 patients with pathologic complete
response (pCR) and 99 patients with non-pCR. 
pCR is a strong surrogate marker for improved disease-free survival and overall survival \citep{spring2020pathologic},
which is defined as no residual invasive cancer in the
breast or lymph nodes, or residual in situ carcinoma without invasive component.
Based on the estimated precision matrix from the gene expression data,
we apply linear discriminant analysis \citep[LDA;][]{MR2722294} to predict patients' pCR status,
and also explore the direct connectivity between gene probesets.

 We transform the MBEI data using \(\log_2(\text{MBEI} + 1)\), making the \(\log_2\)-transformed MBEI approximately Gaussian for each gene probeset. We focus on the 300 most significantly differentially expressed gene probesets selected by the limma moderated \(t\)-test \citep{ritchie2015limma}. We apply a stratified sampling approach to randomly divide the dataset into training and testing sets, with sizes \(n_{\text{train}} = 112\) and \(n_{\text{test}} = 21\), respectively, where the testing set includes 5 pCR subjects and 16 non-pCR subjects (approximately 1/6 of the subjects in each class). This data splitting strategy is randomly repeated 100 times to train and test the LDA.

The LDA assumes that the \(\log_2\)-transformed MBEI data of the \(p = 300\) selected gene probesets follows a Gaussian distribution \(\mathcal{N}(\bd{\mu}_k, \mathbf{\Omega}^{-1})\) with the same precision matrix \(\mathbf{\Omega}\) but different means \(\bd{\mu}_k\), where \(k = 1\) for the pCR class and \(k = 0\) for the non-pCR class. The LDA scores are defined as:
\[
\delta_k(\bd{x}) = \bd{x}^\top \widehat{\mathbf{\Omega}} \widehat{\bd{\mu}}_k - \frac{1}{2} \widehat{\bd{\mu}}_k^\top \widehat{\mathbf{\Omega}} \widehat{\bd{\mu}}_k + \log \widehat{\pi}_k, \quad k = 0, 1,
\]
where \(\bd{x}\) is the data of a given testing subject, \(\widehat{\mathbf{\Omega}}\) is the precision matrix estimated from the sample mean-centered training data, \(\widehat{\bd{\mu}}_k\) is the within-class average vector in the training set, and \(\widehat{\pi}_k\) is the proportion of class-\(k\) subjects in the training set. A testing subject is classified as pCR if \(\delta_1(\bd{x}) > \delta_0(\bd{x})\), and otherwise is classified as non-pCR. 
We select the tuning parameters for each method as in our simulations, except for Nodewise Loreg, for which \(T_{\max,j} = \lfloor n / (\log(p) \log(\log n)) \rfloor\), \(j \in [p]\), as suggested in \citet{zhu2020polynomial}. The nominal FDR level is set to 0.05 for all thresholding estimators of Nodewise Loreg and Nodewise Lasso.

Table~\ref{tab: real data, pCR classify} reports the 
 the classification performance of each precision matrix estimator
 in evaluation metrics including  
 precision, sensitivity, specificity, MCC, and sparsity,  which denotes the
 number of edges in the resulting graph.
 We observe that  the Nodewise Loreg estimators
 $\widehat{\bm{\Omega}}^{\text{S}}$,
 $\mathcal{T}(\widehat{\mb{\Omega}}^{\text{S}}|Z_0(\widehat{\mb{\Omega}}^\text{US}),S_L(\widehat{\mb{\Omega}}^\text{S}))$,
 and
 $\mathcal{T}(\widehat{\mb{\Omega}}^{\text{S}}|Z_0(\widehat{\mb{T}}),S_L(\widehat{\mb{\Omega}}^{\text{S}}))$
 perform similarly and outperform 
 the other estimators, with 
 the highest values in precision, sensitivity, and MCC,
 nearly the highest specificity, and the smallest values in sparsity.

\begin{table}[b!]
 \renewcommand\arraystretch{1.3}
 \begin{center}
\resizebox{\textwidth}{!}{
\begin{tabular}{lccccc}
\toprule
Method & Precision  & Sensitivity & Specificity & MCC & Sparsity: \# edges  \\
\midrule

$L_0{:}~ \widehat{\bm{\Omega}}^{\text{S}}$&{\bf 0.637(0.158)} &{\bf  0.838(0.155)} & 0.829(0.106) &{\bf  0.622(0.164)} & 285.440(10.372)   \\
$L_0{:}~ \mathcal{T}(\widehat{\mb{\Omega}}^{\text{S}}|Z_0(\widehat{\mb{\Omega}}^\text{US}),S_L(\widehat{\mb{\Omega}}^\text{S}))$&{\bf  0.637(0.158)} &{\bf  0.838(0.155)} & 0.829(0.106) & {\bf 0.622(0.164)} &{\bf  284.440(10.372)}  \\
$L_0{:}~  \mathcal{T}(\widehat{\mb{\Omega}}^{\text{S}}|Z_0(\widehat{\mb{T}}),S_L(\widehat{\mb{\Omega}}^{\text{S}}))$&{\bf  0.637(0.158)} &{\bf  0.838(0.155)} & 0.829(0.106) &{\bf  0.622(0.164)} &{\bf  284.440(10.372)} \\
$L_0{:}~  \mathcal{T}(\widehat{\mb{T}}|Z_0(\widehat{\mb{T}}),S_L(\widehat{\mb{\Omega}}^{\text{S}}))$& 0.130(0.085) & 0.310(0.212) & 0.357(0.143) & $-$0.294(0.225) &{\bf  284.440(10.372)} \\
$L_0{:}~  \mathcal{T}(\widehat{\mb{T}}|Z_0(\widehat{\mb{T}}),S_L(\widehat{\mb{T}}))$& 0.210(0.158) & 0.394(0.250) & 0.486(0.197) & $-$0.106(0.329) & 428.410(25.700) \\
$L_0{:}~  \widehat{\mb{T}}$& 0.069(0.053) & 0.206(0.167) & 0.181(0.116) & $-$0.576(0.165) & 44850.00(0.000) \\
$L_1{:}~ \widehat{\mb{\Omega}}^{\text{S}}$& 0.611(0.156) & 0.826(0.172) & 0.813(0.108) &0.592(0.180) & 288.220(15.005) \\ 
$L_1{:}~ \mathcal{T}(\widehat{\mb{\Omega}}^{\text{S}}|Z_0(\widehat{\mb{T}}),S_L(\widehat{\mb{\Omega}}^{\text{S}}))$& 0.612(0.156) & 0.826(0.172) & 0.814(0.108) & 0.593(0.180) & 293.200(15.460) \\ 
$L_1{:}~  \mathcal{T}(\widehat{\mb{T}}|Z_0(\widehat{\mb{T}}),S_L(\widehat{\mb{\Omega}}^{\text{S}}))$& 0.114(0.063) & 0.330(0.198) & 0.233(0.119) & $-$0.409(0.203) & 293.200(15.460) \\
$L_1{:}~  \mathcal{T}(\widehat{\mb{T}}|Z_0(\widehat{\mb{T}}),S_L(\widehat{\mb{T}}))$& 0.051(0.049) & 0.148(0.144) & 0.186(0.110) & $-$0.614(0.158) & 4981.65(565.00) \\
$L_1{:}~  \widehat{\mb{T}}$& 0.051(0.049) & 0.148(0.144) & 0.187(0.106) & $-$0.613(0.153) & 44850.00(0.000) \\
CLIME& 0.608(0.141) & 0.828(0.153) & 0.811(0.107) & 0.591(0.152) & 616.590(41.211) \\
GLasso& 0.627(0.167) & 0.686(0.187) &{\bf  0.855(0.091)} & 0.531(0.176) & 3862.01(353.98) \\
\bottomrule
\end{tabular}}
\caption{Average (standard deviation) of  each pCR classification metric over 100 replications.}
\label{tab: real data, pCR classify}
\end{center}
 \end{table}

 \begin{figure}[b!]
 \begin{subfigure}{0.5\linewidth}
\centering
        \includegraphics[width=0.8\textwidth]{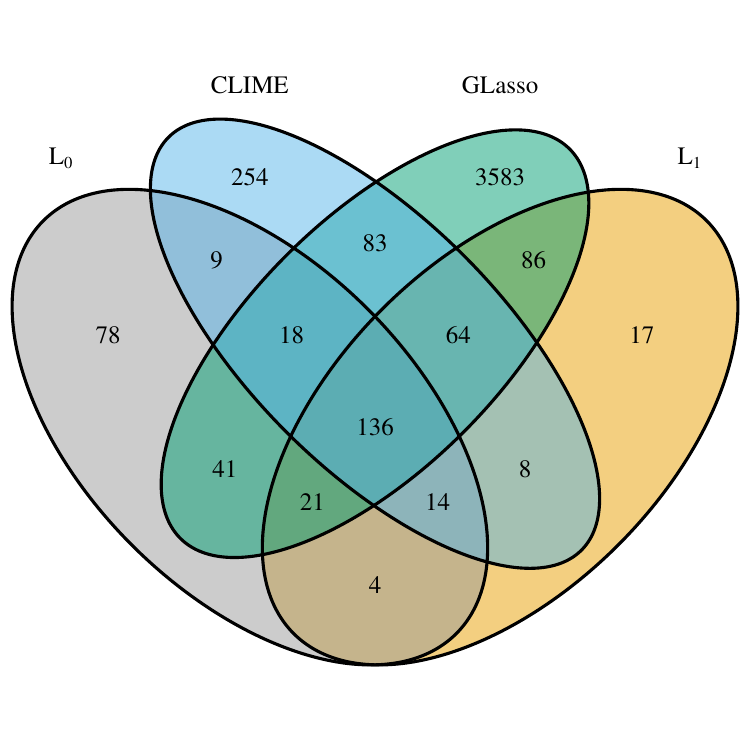}
        \vspace{-0.5cm}
   \caption*{\scriptsize Based on estimated absolute partial correlations} 
     \end{subfigure} 
    \begin{subfigure}{0.5\linewidth}
\centering
        \includegraphics[width=0.8\textwidth]{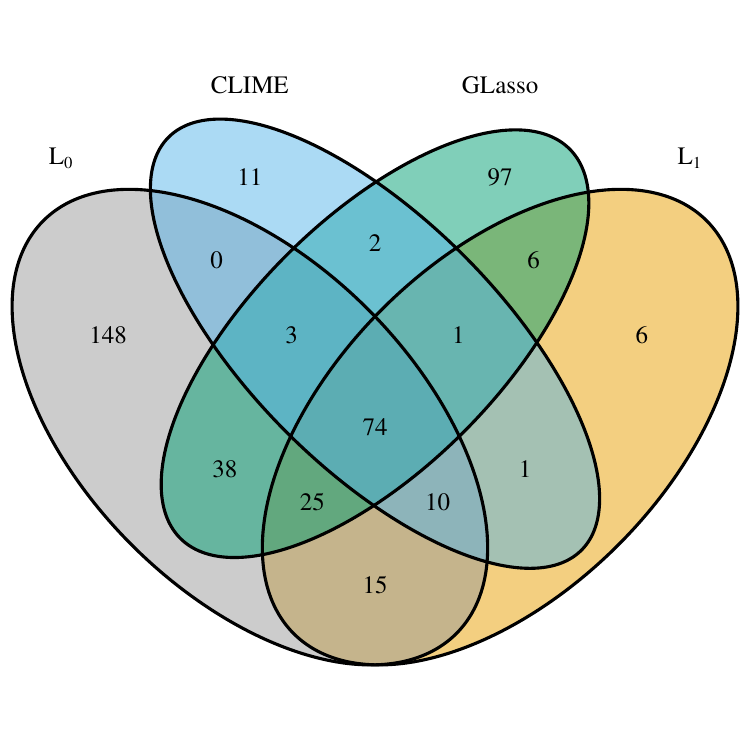}
                \vspace{-0.5cm}
   \caption*{\scriptsize  Based on estimated absolute partial correlations thresholded by 0.1}
    \end{subfigure}  
    \caption{Venn diagrams for direct connections of gene probesets identified by 
    Nodewise Loreg ($L_0{:}\,\mathcal{T}(\widehat{\mb{\Omega}}^{\text{S}}|Z_0(\widehat{\mb{\Omega}}^\text{US}),S_L(\widehat{\mb{\Omega}}^\text{S}))$),
     Nodewise Lasso ($L_1{:}\,\mathcal{T}(\widehat{\mb{\Omega}}^{\text{S}}|Z_0(\widehat{\mb{T}}),S_L(\widehat{\mb{\Omega}}^{\text{S}}))$),
     CLIME, and GLasso.}
        \label{fig: Venn diagram}
\end{figure}

 \begin{figure}[b!]
 \begin{minipage}{1\linewidth}
   \caption*{\small Based on estimated absolute partial correlations}
    \vspace{-0.3cm}
 \begin{subfigure}{0.245\linewidth}
        \centering
        \includegraphics[width=\textwidth]{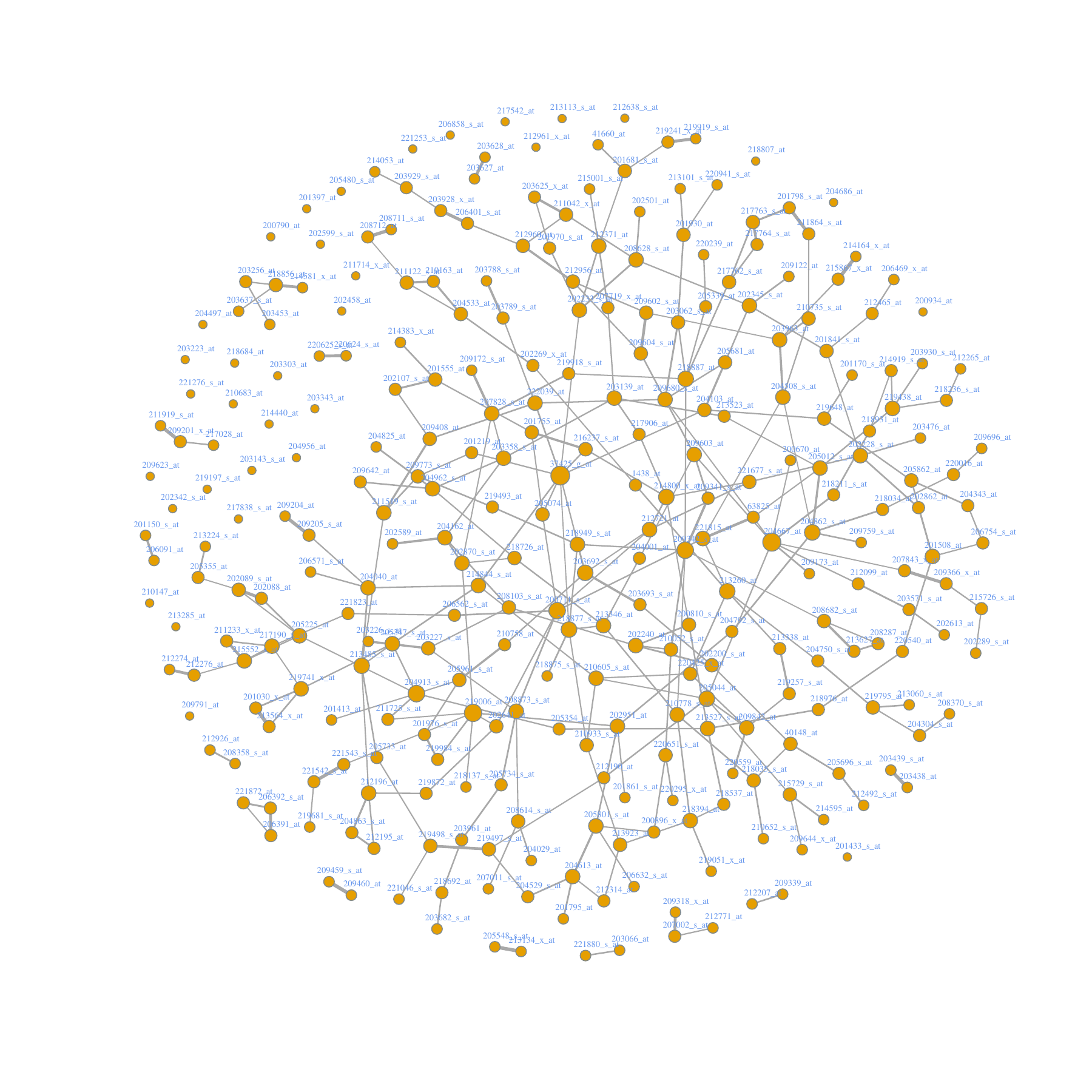}
        \caption*{\scriptsize $L_0{:}~\mathcal{T}(\widehat{\mb{\Omega}}^{\text{S}}|Z_0(\widehat{\mb{\Omega}}^\text{US}),S_L(\widehat{\mb{\Omega}}^\text{S}))$\\
       \centering (321 edges)}
\end{subfigure}   
\begin{subfigure}{0.245\linewidth}
        \centering
        \includegraphics[width=\textwidth]{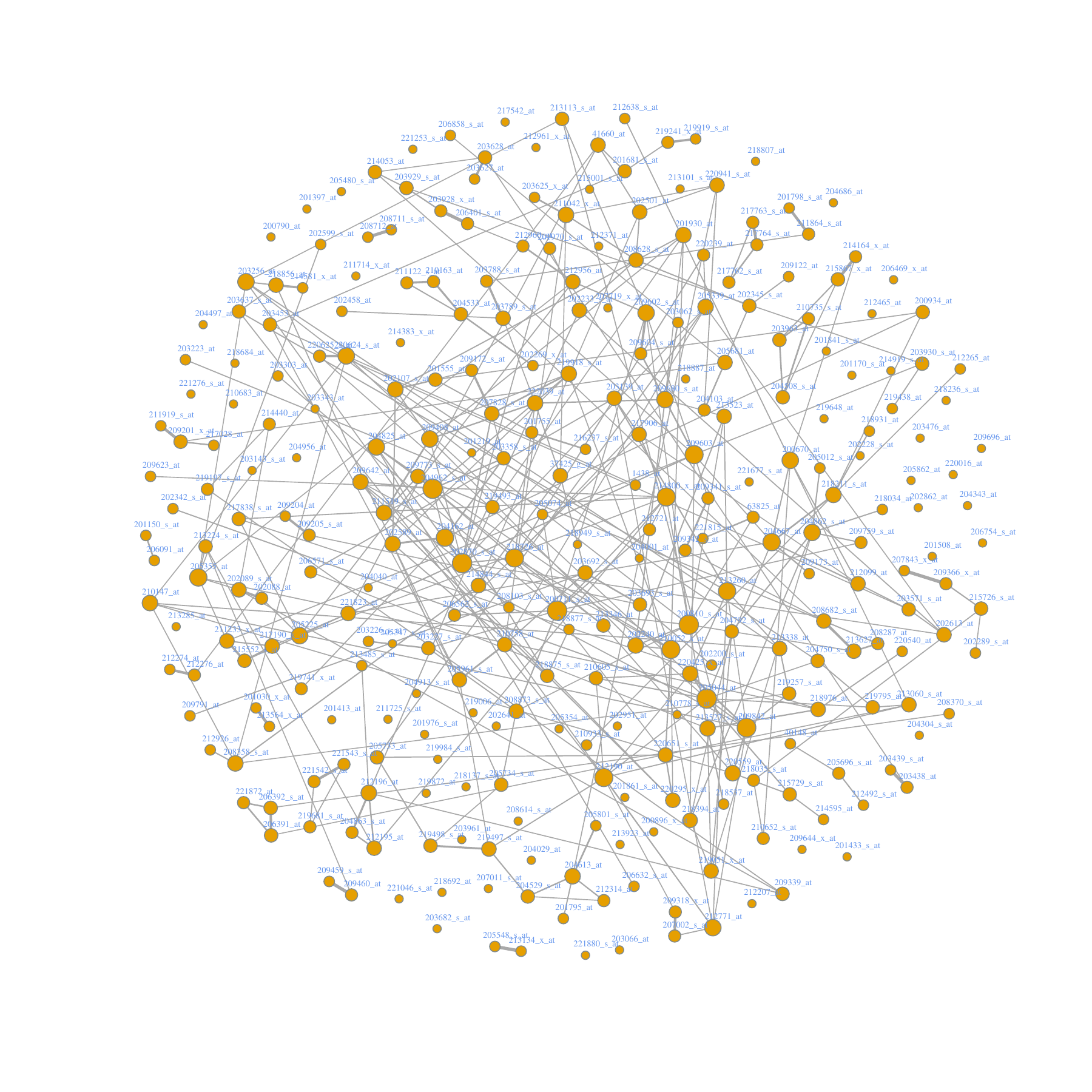}
        \caption*{\scriptsize $L_1{:}~  \mathcal{T}(\widehat{\mb{\Omega}}^{\text{S}}|Z_0(\widehat{\mb{T}}),S_L(\widehat{\mb{\Omega}}^{\text{S}}))$\\
    \centering      (350 edges)}
\end{subfigure}
\begin{subfigure}{0.245\linewidth}
        \centering
        \includegraphics[width=\textwidth]{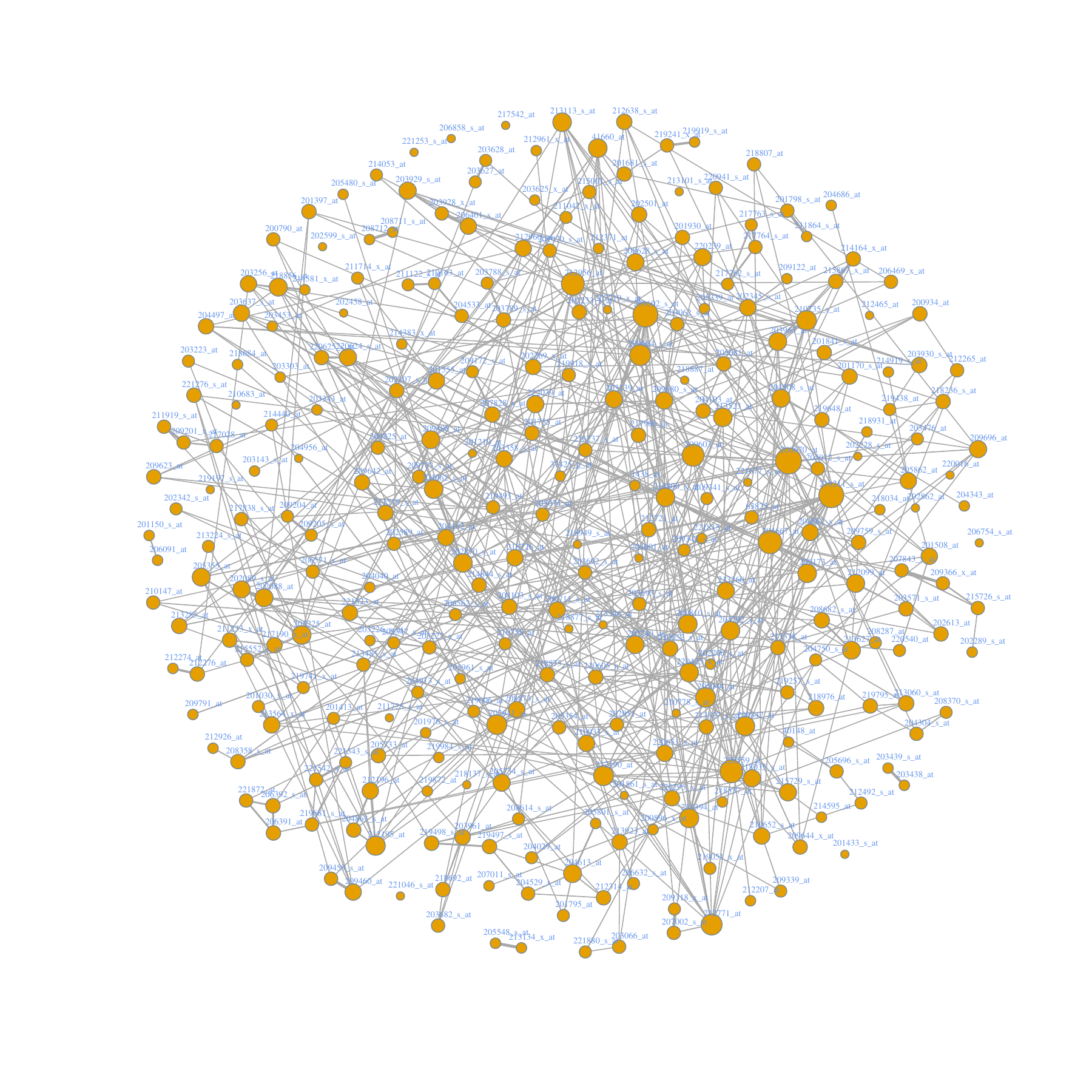}
            \caption*{\scriptsize CLIME\\
            \centering (586 edges)}
\end{subfigure}
\begin{subfigure}{0.245\linewidth}
        \centering
        \includegraphics[width=\textwidth]{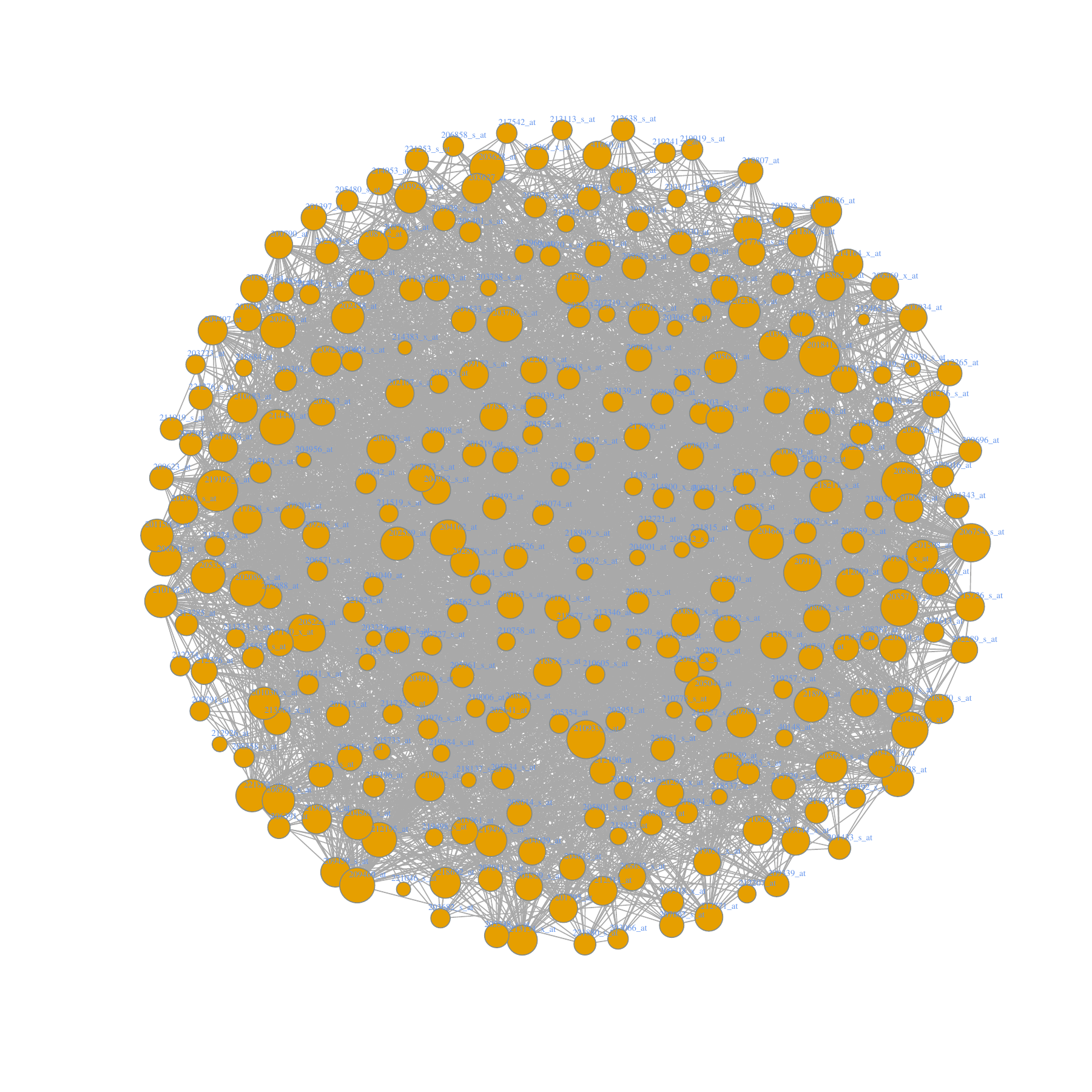}
         \caption*{\scriptsize GLasso\\
         \centering (4032 edges)}
\end{subfigure}
    \end{minipage}
    
     \begin{minipage}{1\linewidth}
   \caption*{\small Based on estimated absolute partial correlations thresholded by 0.1}
\begin{subfigure}{0.245\linewidth}
        \centering
        \includegraphics[width=\textwidth]{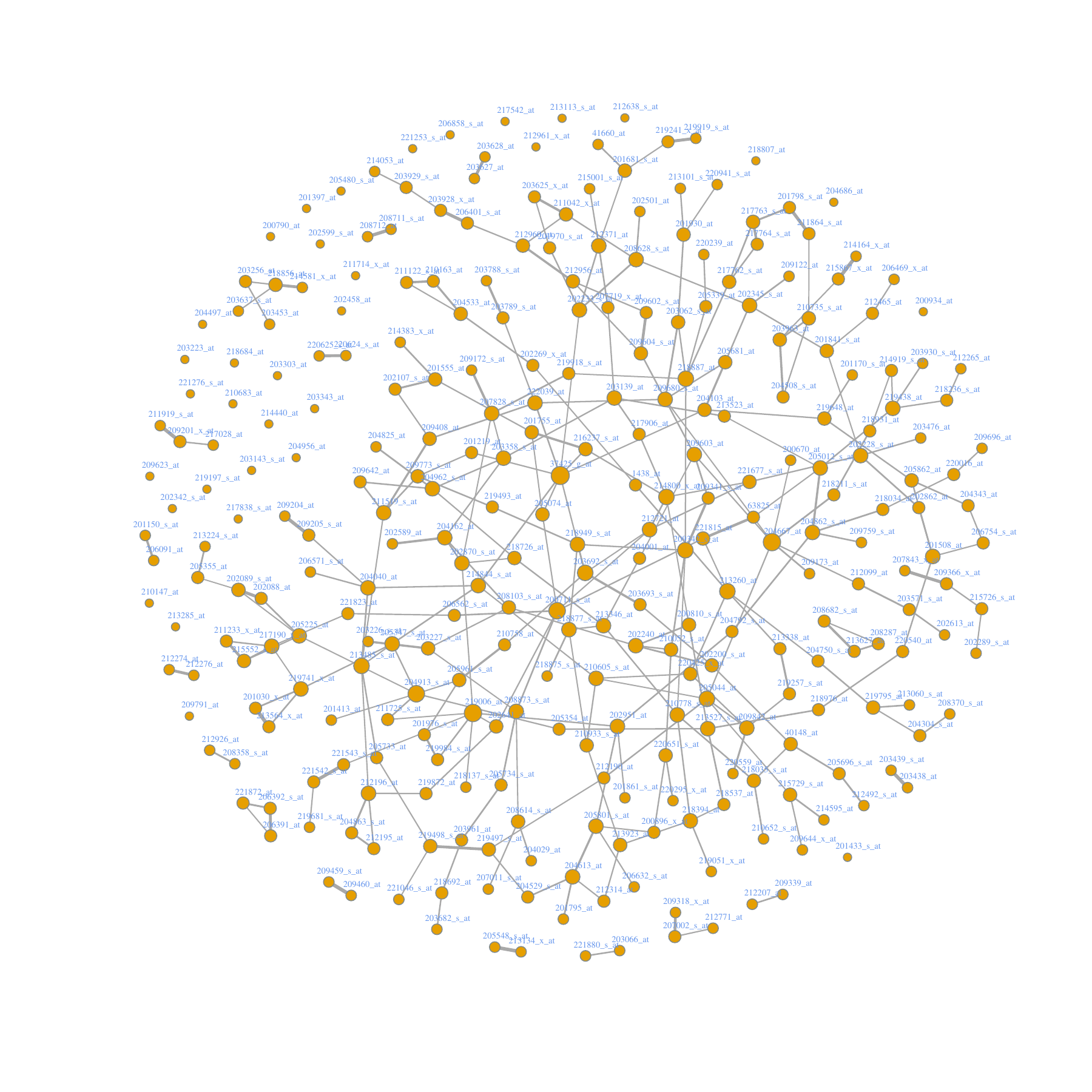}
        \caption*{\scriptsize $L_0{:}~\mathcal{T}(\widehat{\mb{\Omega}}^{\text{S}}|Z_0(\widehat{\mb{\Omega}}^\text{US}),S_L(\widehat{\mb{\Omega}}^\text{S}))$\\
        \centering (313 edges)}
\end{subfigure} 
\begin{subfigure}{0.245\linewidth}
        \centering
        \includegraphics[width=\textwidth]{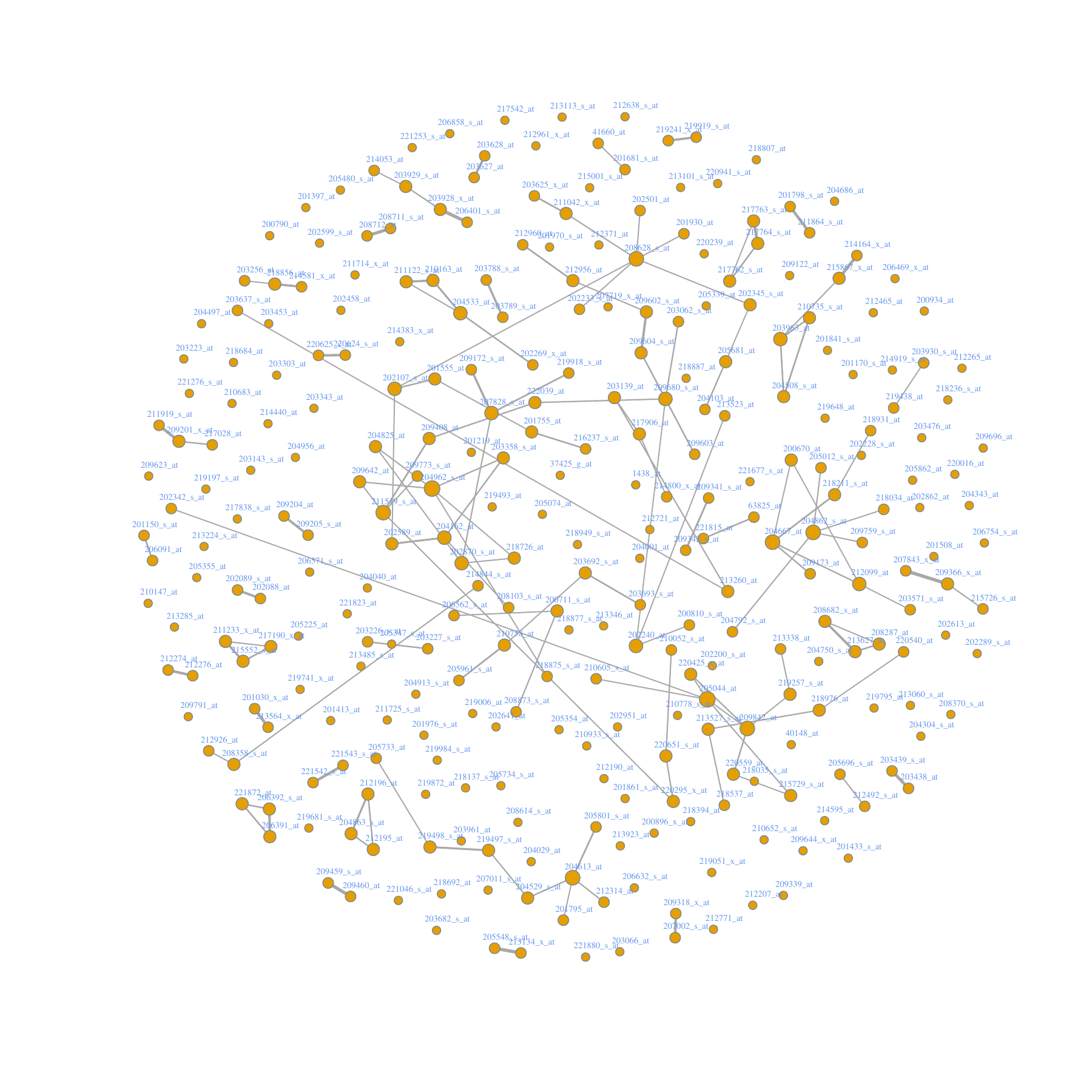}
        \caption*{\scriptsize $L_1{:}~  \mathcal{T}(\widehat{\mb{\Omega}}^{\text{S}}|Z_0(\widehat{\mb{T}}),S_L(\widehat{\mb{\Omega}}^{\text{S}}))$\\
        \centering (138 edges)}
\end{subfigure}    
\begin{subfigure}{0.245\linewidth}
        \centering
        \includegraphics[width=\textwidth]{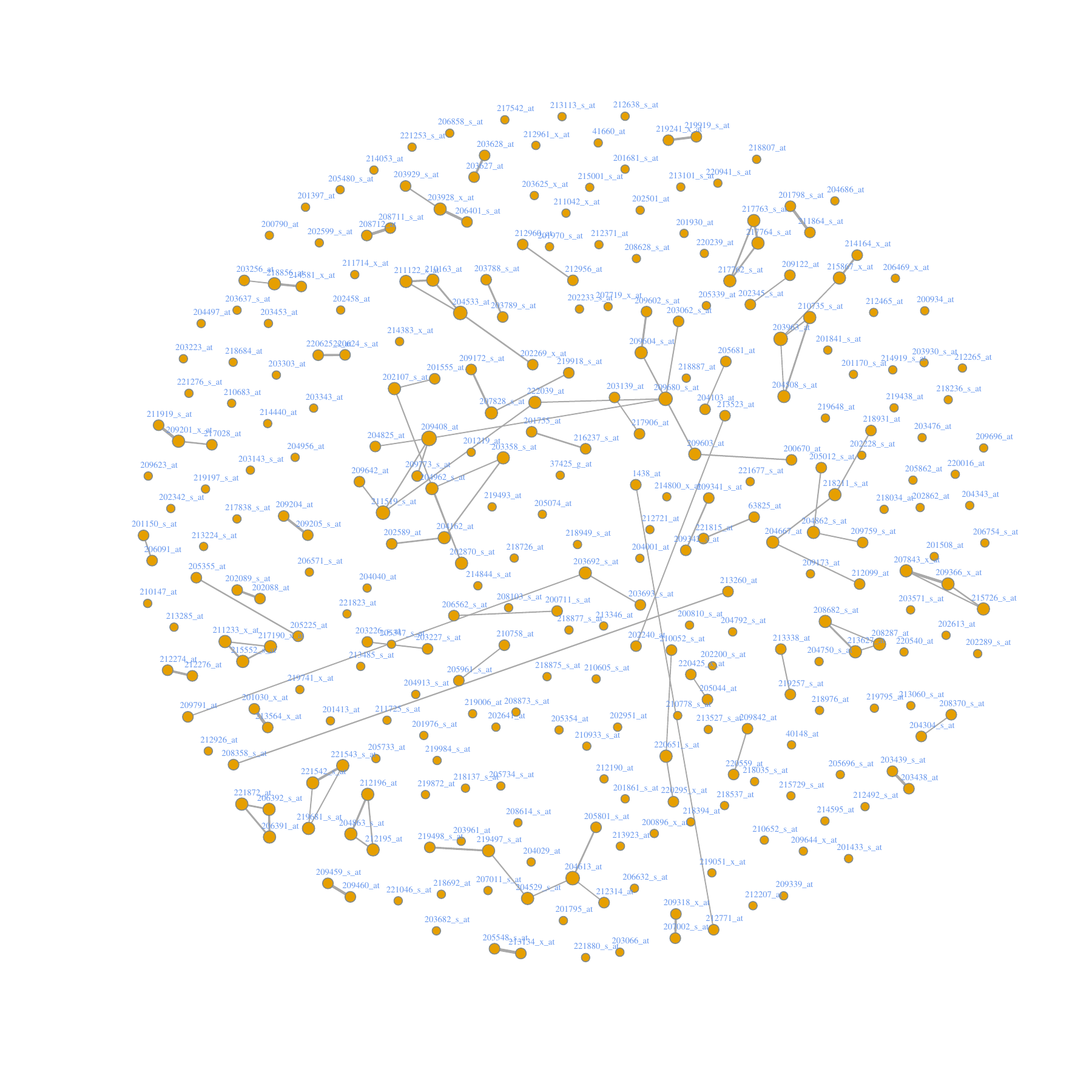}
                \caption*{\scriptsize CLIME\\
                \centering (102 edges)}
\end{subfigure}    
\begin{subfigure}{0.245\linewidth}
        \centering
        \includegraphics[width=\textwidth]{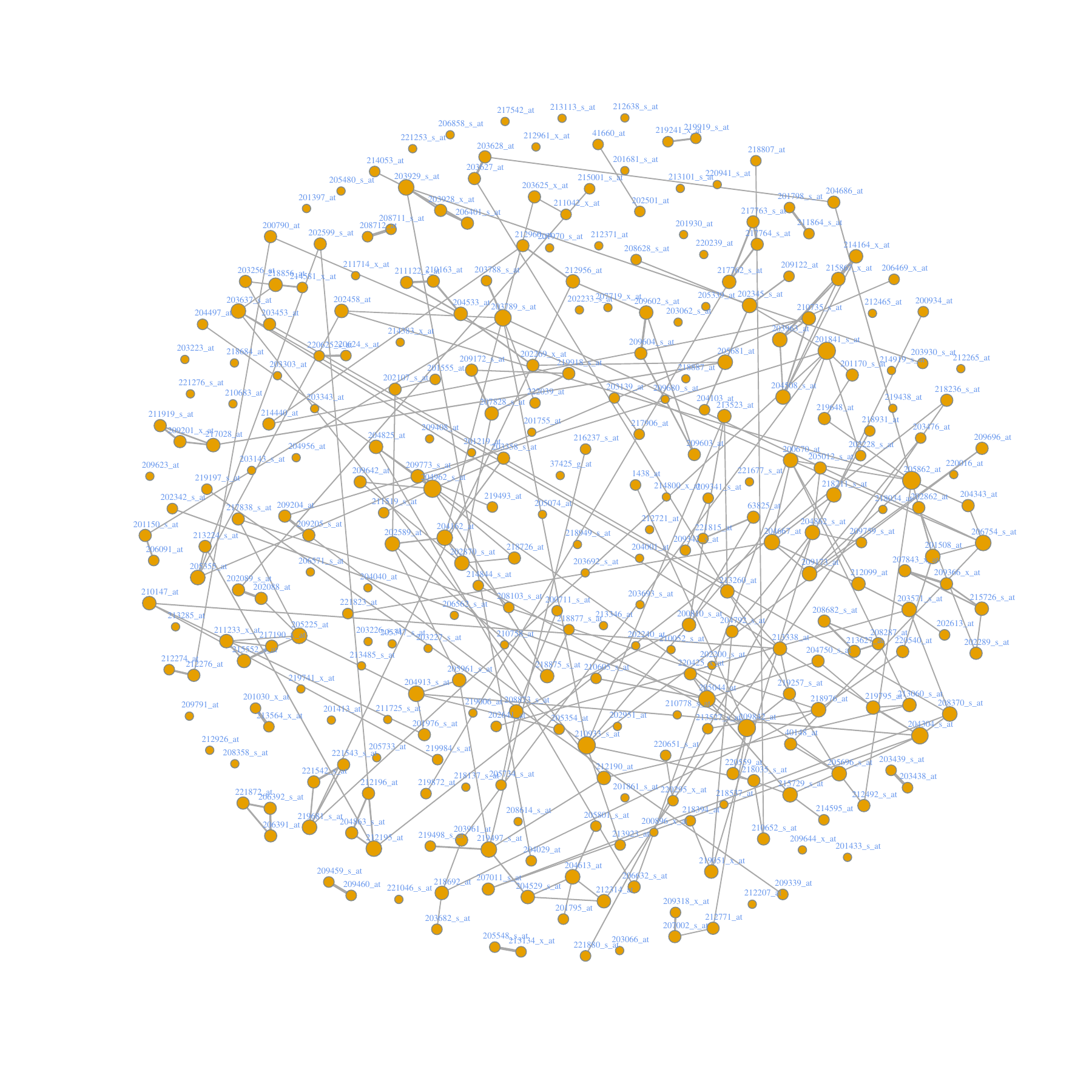}
     \caption*{\scriptsize GLasso\\
                \centering (246 edges)}
\end{subfigure}   
    \end{minipage}
    \caption{Network graphs of gene probesets identified by 
    Nodewise Loreg ($L_0{:}\,\mathcal{T}(\widehat{\mb{\Omega}}^{\text{S}}|Z_0(\widehat{\mb{\Omega}}^\text{US}),S_L(\widehat{\mb{\Omega}}^\text{S}))$),
     Nodewise Lasso ($L_1{:}\,\mathcal{T}(\widehat{\mb{\Omega}}^{\text{S}}|Z_0(\widehat{\mb{T}}),S_L(\widehat{\mb{\Omega}}^{\text{S}}))$),
     CLIME, and GLasso.}
    \label{fig: network of probesets}
\end{figure}

We analyze the direct connectivity among the 300 differentially expressed gene probesets using the estimated partial correlation matrices derived from the four precision matrix estimation methods based on the entire dataset. The partial correlation between \(x_i\) and \(x_j\) given \(\{x_k\}_{k \notin \{i,j\}}\) is equal to \(-\mathbf{\Omega}_{ij}/\sqrt{\mathbf{\Omega}_{ii}\mathbf{\Omega}_{jj}}\) \citep{Cram46}. A zero partial correlation indicates conditional independence if \(\bd{x}\) is Gaussian or nonparanormal \citep{Liu09}.

We focus on \(\mathcal{T}(\widehat{\mathbf{\Omega}}^{\text{S}}|Z_0(\widehat{\mathbf{\Omega}}^{\text{US}}),S_L(\widehat{\mathbf{\Omega}}^{\text{S}}))\) for Nodewise Loreg and \(\mathcal{T}(\widehat{\mathbf{\Omega}}^{\text{S}}|Z_0(\widehat{\mathbf{T}}),S_L(\widehat{\mathbf{\Omega}}^{\text{S}}))\) for Nodewise Lasso due to their superior performance in pCR classification. We detect 321, 350, 586, and 4032 direct connections of gene probesets using Nodewise Loreg, Nodewise Lasso, CLIME, and GLasso, respectively (Figure~\ref{fig: Venn diagram}).

Nodewise Loreg uniquely identifies 78 direct connections not found by the other methods, while 4095 direct connections are identified by at least one other method but not by Nodewise Loreg, with 3583 unique to GLasso. Of these 4095 connections, only 14 have an absolute partial correlation estimate larger than 0.1 by at least two other methods, and only one of these is detected by
another Nodewise Loreg  estimator
$\mathcal{T}(\widehat{\mb{T}}|Z_0(\widehat{\mb{T}}),S_L(\widehat{\mb{T}}))$. Among the 78 unique connections identified by Nodewise Loreg, 49 are also found by another Nodewise Lasso estimator
$\mathcal{T}(\widehat{\mb{T}}|Z_0(\widehat{\mb{T}}),S_L(\widehat{\mb{T}}))$, and only 5, which are in the remaining 29, have an absolute partial correlation estimate no larger than 0.1.

Figure~\ref{fig: network of probesets} shows network graphs of the 300 gene probesets based on estimated absolute partial correlations from each method. We observe that 60.6\%, 82.6\%, and 93.9\% of the direct connections (i.e., edges) detected by Nodewise Lasso, CLIME, and GLasso, respectively, are weak (estimated absolute partial correlations \(\leq 0.1\)), while only 8 direct connections found by Nodewise Loreg are weak. For networks thresholded at 0.1, the majority of Nodewise Lasso and CLIME’s structures (89.9\% and 85.3\% of edges) are subsumed within the network identified by Nodewise Loreg, which significantly differs from GLasso’s network, covering only 56.9\% of its edges. 
Given GLasso's poor performance in simulations, we focus on Nodewise Loreg's network. The unthresholded network from Nodewise Loreg contains 33 isolates, 9 dyads, 3 triads, and 1 pentad, with the remaining 235 gene probesets forming a connected subnetwork with several small clusters and no hubs, as the maximum node degree is 9.

%Table~\ref{tab: top 10 probes} lists the top 10 pairs of directly connected probesets and their genes with largest absolute partial correlations estimated from Nodewise Loreg. All these 10 pairs are  also identified by the other three methods with high absolute partial correlations, and the two probesets within the same pair are on the same gene. 

\section{Conclusion}\label{s: conclusion}

In this paper, we propose Nodewise Loreg,  a nodewise $L_0$-penalized regression method for estimating high-dimensional sparse precision matrix. 
We establish its asymptotic results, including convergence rates in various matrix norms, support recovery, and asymptotic normality, under high-dimensional sub-Gaussian settings.
In particular, 
under certain conditions, we
prove that the Nodewise Loreg estimator
is asymptotically unbiased and normally distributed without the need for debiasing.
A desparsified version of the Nodewise Loreg estimator is also developed in parallel to 
the desparsified Nodewise Lasso estimator.
The asymptotic
variances of undesparsified  Nodewise Loreg estimator
are entrywise dominated by those of the two desparsified estimators
for Gaussian data,
potentially providing more powerful statistical inference when the data is Gaussian or can be transformed to be nearly Gaussian.
Our extensive simulations demonstrate that the undesparsified Nodewise Loreg estimator generally surpasses the
two desparsified estimators in exhibiting asymptotic normal behavior.
Moreover, the proposed Nodewise Loreg method 
outperforms Nodewise Lasso, CLIME, and GLasso
in most simulation settings in terms of matrix norm losses, support recovery, and timing performance.
The analysis of MDA133 breast cancer gene expression dataset
shows that our Nodewise Loreg performs better in the pCR classification
and detects a higher quality partial correlation network of gene probesets than the three $L_1$-norm based methods.

There are many possible future extensions of our research. 
The first is to extend 
 the entrywise asymptotic normality of 
 the Nodewise Loreg estimator to 
the asymptotic normality of a subset of 
its matrix entries.
The subset-based asymptotic normality
has been recently studied by 
\citet{chang2018confidence} and \citet{klaassen2023uniform}
for Nodewise Lasso-type estimators,
which enables 
simultaneous inference 
via joint hypothesis testing
without multiple testing correction.
Another direction is to extend our theory 
from i.i.d. data to time series data. 
In many applications,  
such as brain functional connectivity analysis using functional MRI \citep{ryali2012estimation},
data are not a i.i.d. random sample but are temporally dependent observations.
\citet{Chen13,chen2016regularized}, \citet{Shu19}, and \citet{zhang2021minimax} 
show that
large precision matrix estimators like GLasso and CLIME
are still applicable to high-dimensional time series data, even with strong temporal dependence,
and provide
 convergence rates 
depending on the strength of temporal dependence.
These interesting and challenging studies are worth further investigation.

%1. approximately sparse

%2. conference regions instead of entrywise AN

%3. time series data.

%Supplement

\section*{Supplementary Materials}

\section{Theoretical proofs}\label{sec: proofs}

\begin{lem}\label{lemma: bound for var(e)}
	Define $\epsilon_j=x_j-\bd{x}_{\setminus j}^\top\bd{\alpha}_j^*$,  $\bd{\epsilon}_j=\mb{X}_{*j}-\mb{X}_{*\setminus j}\bd{\alpha}_j^*$, and $\sigma_j^2=\var(\epsilon_j)$. If Assumptions~\ref{Theta bound} and \ref{subGauss} hold
	and $(\log p)/n=o(1)$, then
	$$
P\left(	\max_{j\in[p]} \left|\|\bd{\epsilon}_j \|_2^2/n-\sigma_j^2\right|> M \sqrt{(\log p)/n}  \right)=O(p^{-C})
	$$
for any given constant $C>0$
	and some constant $M>0$ dependent on $C$.	
\end{lem}	
\begin{proof}[Proof of Lemma~\ref{lemma: bound for var(e)}]
In the rest of the paper, for simplicity, we choose a constant $\kappa>1$ such that
$\kappa^{-1}\le \kappa_1^{-1}\le \kappa_2\le\kappa$.
We first prove $\epsilon_j$	is a sub-Gaussian variable.
Define $\bd{v}_j=(1, -\bd{\alpha}_j^{*\top})^\top$.
Recall that $\bd{\alpha}^*_j=\mb{\Sigma}^{-1}_{\setminus j,\setminus j}\mb{\Sigma}_{\setminus j, j}$.
By the second inequality in Lemma~1 of \citet{Lam09},
\[
\|\bd{\alpha}_j^*\|_2=\|\mb{\Sigma}^{-1}_{\setminus j,\setminus j}\mb{\Sigma}_{\setminus j, j}\|_2\le
\|\mb{\Sigma}^{-1/2}_{\setminus j,\setminus j}\|_2\|\mb{\Sigma}^{-1/2}_{\setminus j,\setminus j}\mb{\Sigma}_{\setminus j, j}\|_2= \lambda_{\min}^{-1/2}(\mb{\Sigma}_{\setminus j,\setminus j} ) (\mb{\Sigma}_{j,\setminus j}\mb{\Sigma}^{-1}_{\setminus j,\setminus j}\mb{\Sigma}_{\setminus j,j})^{1/2}.
\]
By Cauchy’s interlace theorem \citep{hwang2004cauchy} and Assumption~\ref{Theta bound},
\[
\lambda_{\min}^{-1/2}(\mb{\Sigma}_{\setminus j,\setminus j} )
\le \lambda_{\min}^{-1/2}(\mb{\Sigma})=\lambda_{\max}^{1/2}(\mb{\Omega})\le \kappa^{1/2}.
\]
From \eqref{Theta_jj formula}, we have
$$\mb{\Sigma}_{j,\setminus j} \mb{\Sigma}_{\setminus j,\setminus j}^{-1} \mb{\Sigma}_{\setminus j,j}=\mb{\Sigma}_{jj}-\mb{\Omega}_{jj}^{-1}\le \max_{j\in[p]}\mb{\Sigma}_{jj}-1/\max_{j\in[p]}\mb{\Omega}_{jj}\le  \lambda_{\max}(\mb{\Sigma})-1/\lambda_{\max}(\mb{\Omega})\le \kappa-\kappa^{-1}.$$
From the above,
we obtain
\be\label{norm-2 of alpha_j}
\|\bd{\alpha}_j^*\|_2\le 
\lambda_{\min}^{-1/2}(\mb{\Sigma}_{\setminus j,\setminus j} ) (\mb{\Sigma}_{j,\setminus j}\mb{\Sigma}^{-1}_{\setminus j,\setminus j}\mb{\Sigma}_{\setminus j,j})^{1/2}
\le \kappa^{1/2}(\kappa-\kappa^{-1})^{1/2}
=(\kappa^2-1)^{1/2}
\ee
Hence, $\| \bd{v}_j \|_2^2=\| (1,-\bd{\alpha}_j^{*\top})^\top\|_2^2
=1+\|\bd{\alpha}_j^*\|_2^2\le 1+(\kappa^2-1)=\kappa^2$.
Then from Assumption~\ref{subGauss},
we have
\begin{align}\label{subGauss epsilon}
E\left[\exp(\epsilon_j^2/(\kappa K)^2)\right]
&=E\left[\exp(|\bd{v}_j^\top(x_j,\bd{x}_{\setminus j}^\top)^\top|^2/(\kappa K)^2)\right]
=E\left[\exp(|\kappa^{-1}\bd{v}_j^\top(x_j,\bd{x}_{\setminus j}^\top)^\top|^2/K^2)\right]\nonumber\\
&\le \sup_{\bd{v}\in \mathbb{R}^p:\|\bd{v}\|_2\le 1}E\left[\exp(|\bd{v}^\top\bd{x}|^2/K^2)\right]\le 2. 
\end{align}
Therefore, $\epsilon_j$ is a sub-Gaussian random variable \citep[][Definition 2.5.6]{vershynin_2018} with a parameter $\kappa K$ for all $j\in[p]$.

From Theorem~1.1 
in \cite{Rude13},
for any given constant $C>0$, there exists a constant $M>0$ such 
that 
\be\label{ineq: max e cencentration}
P\left(\max_{j\in[p]}\left|\frac{\|\bd{\epsilon}_j\|_2^2}{n}-\sigma_j^2 \right|>M\sqrt{\frac{\log p}{n}}\right)\le \sum_{j=1}^pP\left(\left|\frac{\|\bd{\epsilon}_j\|_2^2}{n}-\sigma_j^2 \right|>M\sqrt{\frac{\log p}{n}}\right)
=O(p^{-C}).
\ee

The proof is complete.
\end{proof}

\begin{proof}[Proof of Theorem~\ref{thm: consistency}]
We first consider the proof under Assumption~\ref{(C1)}.

We will use Theorem 4 (i) in \citet{Huang18} 
to obtain the bound of $\max_{j\in[p]}\| \widehat{\bd{\beta}}_j-\bd{\beta}_j^*\|_2$, where $\bd{\beta}_j^*:=
\widehat{\mb{\Gamma}}_{\setminus j, \setminus j}^{1/2}\bd{\alpha}_j^*$.
Before using the theorem, 
we need to verify whether
their Assumptions (A1) and (A2) hold in our case.  
Due to our Assumption~\ref{assmp: bounded Tj},
their (A1) holds for all $\{T_i\}_{i=1}^p$.
Next, we check their (A2).
We have that 
\begin{align}
\max_{1\le \ell\le |A|, A\in S_{2T}}|\lambda_\ell(\widehat{\mb{\Sigma}}_{AA}) 
-\lambda_\ell( \mb{\Sigma}_{AA})|
&\le \max_{A\in S_{2T}}\|  \widehat{\mb{\Sigma}}_{AA} - \mb{\Sigma}_{AA} \|_2
\le  \max_{A\subseteq [p], |A|=2T} \|  \widehat{\mb{\Sigma}}_{AA} - \mb{\Sigma}_{AA} \|_2
\nonumber\\
&\le 
 \check{M}_0 \sqrt{T (\log p)/n}       
\label{Sigma_AA's error in norm 2}
\end{align}
with probability $1-O(p^{-C})$ for 
any given constant $C>0$ and some constant 
$\check{M}_0>0$ dependent on $C$,
where
the first inequality follows from
Weyl's inequality, the second inequality follows from Cauchy’s interlace theorem,
and the last inequality 
is obtained by using Assumption~\ref{subGauss}, $T(\log p)/n=o(1)$,
Remark 5.40 in \citet{Vershynin2012} (in which let $t=\sqrt{(\log p^{2T+C})/c}$), and the same proof technique used in \eqref{ineq: max e cencentration} (at most $p^{2T}$ terms in the summation).
From Cauchy’s interlace theorem %\citep{hwang2004cauchy} 
and Assumption~\ref{Theta bound},
for any $A\subseteq [p]$,
\be \label{bounded Sigma_AA}
\kappa^{-1}\le \lambda_{\min}(\mb{\Sigma}) \le \lambda_{\min}(\mb{\Sigma}_{AA})\le \lambda_{\max}( \mb{\Sigma}_{AA})\le \lambda_{\max}(\mb{\Sigma})
\le \kappa.
\ee
By \eqref{Sigma_AA's error in norm 2}, \eqref{bounded Sigma_AA},
and $T(\log p)/n=o(1)$,
we have
\be\label{spectrum of Sigma_AA}
P\left((2\kappa)^{-1}\le 
\lambda_{\min}(\widehat{\mb{\Sigma}}_{AA}) \le \lambda_{\max}(\widehat{\mb{\Sigma}}_{AA}) 
\le 2\kappa, \forall A\in S_{2T} \right)=1-O(p^{-C}).
\ee
Then, it holds with probability $1-O(p^{-C})$ that 
 \begin{align}\label{lb min Sigma_ii}
  (2\kappa)^{-1}\le\min_{A\in S_{2T}} \lambda_{\min}(\widehat{\mb{\Sigma}}_{AA})
 & \le\min_{i\in A,A\in S_{2T}} \widehat{\mb{\Sigma}}_{ii}\nonumber\\
 &=\min_{i\in [p]}\widehat{\mb{\Sigma}}_{ii}
\le \max_{i\in [p]}\widehat{\mb{\Sigma}}_{ii}\nonumber\\
&\qquad\qquad~~=
\max_{i\in A,A\in S_{2T}} \widehat{\mb{\Sigma}}_{ii}
\le \max_{A\in S_{2T}} \lambda_{\max}(\widehat{\mb{\Sigma}}_{AA})\le 2\kappa.
 \end{align}
 From \eqref{lb min Sigma_ii},
$P(\min_{i\in[p]}\widehat{\mb{\Sigma}}_{ii}\ge (2\kappa)^{-1}> 0)=1-O(p^{-C})$,
which implies 
\be\label{Gamma exists}
P(\widehat{\mb{\Sigma}}_{ii}^{-1}~\text{exists}, \forall i\in[p]) =1-O(p^{-C}).
\ee
From Assumption~\ref{Theta bound}, we have
 \be\label{lb for sigma_ii}
 \kappa_2^{-1}\le\lambda_{\min}(\mb{\Sigma})
 \le \min_{i\in[p]} \mb{\Sigma}_{ii}
 \le \max_{i\in[p]} \mb{\Sigma}_{ii}
 \le \lambda_{\max}(\mb{\Sigma}) 
\le \kappa_1 .
\ee 
From the sub-Gaussian concentration inequality in (29) of \citet{CaiClime},
Assumption~\ref{subGauss}, and
$(\log p)/n=o(1)$,
\be\label{max norm err of Sigma}
\| \widehat{\mb{\Sigma}}-\mb{\Sigma}\|_{\max}\le \tilde{M}_0\sqrt{(\log p)/n}
\ee
with probability $1-O(p^{-C})$ for some constant $\tilde{M}_0>0$ dependent on $C$.
By the mean value theorem, \eqref{lb min Sigma_ii}, \eqref{lb for sigma_ii},
and \eqref{max norm err of Sigma},
we have 
\begin{align}\label{max err in Sigma^-1/2}
\max_{i\in[p]}|\widehat{\mb{\Sigma}}_{ii}^{-1/2}-\mb{\Sigma}_{ii}^{-1/2}|
\le \frac{1}{2}\max_{i\in[p]}\{  \widehat{\mb{\Sigma}}_{ii}^{-3/2},  \mb{\Sigma}_{ii}^{-3/2} \}\max_{i\in[p]}|\widehat{\mb{\Sigma}}_{ii}-\mb{\Sigma}_{ii}|
\le \check{M}_1 \sqrt{(\log p)/n}
\end{align}
with probability $1-O(p^{-C})$ for some constant $\check{M}_1>0$ dependent on $C$.
By Weyl’s inequality, \eqref{Sigma_AA's error in norm 2}, \eqref{bounded Sigma_AA}, 
\eqref{lb min Sigma_ii}, \eqref{lb for sigma_ii}, and \eqref{max err in Sigma^-1/2},
we have
\begin{align}\label{diff in eigen sample corr}
&\max_{1\le \ell\le |A|, A\in S_{2T}}|\lambda_\ell(\mb{Z}_{*A}^\top\mb{Z}_{*A}/n) 
-\lambda_\ell( \mb{R}_{AA})|\nonumber\\
&\le
\max_{A\in S_{2T}}\| \mb{Z}_{*A}^\top\mb{Z}_{*A}/n
-\mb{R}_{AA}\|_2\nonumber\\
&=\max_{A\in S_{2T}}\|\widehat{\mb{\Gamma}}_{AA}^{-1/2}\widehat{\mb{\Sigma}}_{AA}\widehat{\mb{\Gamma}}_{AA}^{-1/2}-   \mb{\Gamma}_{AA}^{-1/2}\mb{\Sigma}_{AA}\mb{\Gamma}_{AA}^{-1/2}\|_2
\nonumber\\
&\le \max_{A\in S_{2T}}\Big\{
\|\widehat{\mb{\Gamma}}_{AA}^{-1/2}(\widehat{\mb{\Sigma}}_{AA}
-\mb{\Sigma}_{AA})\widehat{\mb{\Gamma}}_{AA}^{-1/2}\|_2
+\|\widehat{\mb{\Gamma}}_{AA}^{-1/2}\mb{\Sigma}_{AA}(\widehat{\mb{\Gamma}}_{AA}^{-1/2}-\mb{\Gamma}_{AA}^{-1/2})\|_2
\nonumber\\
&\qquad\qquad\qquad+\|(\widehat{\mb{\Gamma}}_{AA}^{-1/2}-\mb{\Gamma}_{AA}^{-1/2})
\mb{\Sigma}_{AA}\mb{\Gamma}_{AA}^{-1/2}
 \|_2\Big\}\nonumber\\
 &\le \tilde{M}_1 \sqrt{T(\log p)/n}
\end{align}
with probability $1-O(p^{-C})$ for some constant $\tilde{M}_1>0$ dependent on $C$,
where $\mb{R}:=\corr(\bd{x})$.
Since
$\lambda_{\max}(\mb{\Sigma})\mb{\Gamma}^{-1} \succeq
\mb{R} =\mb{\Gamma}^{-1/2} \mb{\Sigma}\mb{\Gamma}^{-1/2}
\succeq 
\lambda_{\min}(\mb{\Sigma})\mb{\Gamma}^{-1} $,
where $\mb{M}_1\succeq \mb{M}_2$ denotes that $\mb{M}_1-\mb{M}_2$ is a positive semi-definite matrix, then
together with \eqref{lb for sigma_ii}
we obtain
\be\label{eigen bounds for R}
 (\kappa_1\kappa_2)^{-1}
\le \lambda_{\min}(\mb{\Sigma})(\max_{i\in[p]} \mb{\Sigma}_{ii})^{-1}
\le
\lambda_{\min}(\mb{R})\le 
\lambda_{\max}(\mb{R})
\le 
\lambda_{\max}(\mb{\Sigma})(\min_{i\in [p]} \mb{\Sigma}_{ii})^{-1}
\le \kappa_1\kappa_2.
\ee
By Cauchy’s interlace theorem,
$
 (\kappa_1\kappa_2)^{-1}\le
 \lambda_{\min}(\mb{R})\le  \lambda_{\min}(\mb{R}_{AA})
 \le \lambda_{\max}(\mb{R}_{AA})\le
\lambda_{\max}(\mb{R})
 \le \kappa_1\kappa_2
$
for any $A\subseteq [p]$.
Combining it with \eqref{diff in eigen sample corr} and $T(\log p)/n=o(1)$ yields
\begin{align}\label{P(Event1)}
P(\mathcal{E}_{\widehat{\mb{R}}})&=1-O(p^{-C})\quad \text{where}\\
&\mathcal{E}_{\widehat{\mb{R}}}:=\{\tilde{c}_\kappa^{-1}\le
\lambda_{\min}(\mb{Z}_{*A}^\top\mb{Z}_{*A}/n)\le 
\lambda_{\max}(\mb{Z}_{*A}^\top\mb{Z}_{*A}/n)\le \tilde{c}_\kappa,
\forall A\in S_{2T}
\}\nonumber
\end{align}
with $\tilde{c}_\kappa=c_\kappa+o(1)>c_\kappa:=\kappa_1\kappa_2$.
Hence, Assumption (A2) in \citet{Huang18} holds for 
all $\{T_i\}_{i=1}^p$
with probability $1-O(p^{-C})$.
Recall from Algorithms~\ref{alg: SDAR} and \ref{alg: Nodewise Loreg} that 
$\widehat{\bd{\beta}}_j= \bd{\beta}_j^{k_j+1}$
and
$\widehat{A}_j=A_j^{k_j}$.
Then by Theorem 4 (i) in \citet{Huang18},
on the event $\mathcal{E}_{\widehat{\mb{R}}}$,
for
any
$
\tilde{\theta}_{T,T}\ge  \max_{\substack{A,B\in S_{T}: A\cap B=\emptyset}}\| \mb{Z}_{*A}^\top\mb{Z}_{*B}/n\|_2
$,
we have that, for all $j\in[p]$, 
\begin{align}
\|  \bd{\beta}_j^{k_j+1}-
\bd{\beta}_j^* \|_2
&\le \tilde{b}_1 \tilde{\gamma}_T^{k_j} \| \bd{\beta}_j^*  \|_2
+\tilde{b}_2 h_2(T_j),\label{error in beta_k+1}\\
\| (\bd{\beta}_j^*)_{A_j^*\setminus A_j^{k_j}}\|_2 &
\le \tilde{\gamma}_T^{k_j} \| \bd{\beta}_j^*  \|_2
+\frac{\tilde{\gamma}_T^{}}{(1-\tilde{\gamma}_T^{})\tilde{\theta}_{T,T}}h_2(T_j),\label{norm of beta outside}
\end{align}
when
$\tilde{\gamma}_T^{}:=[2\tilde{\theta}_{T,T}+(1+\sqrt{2})\tilde{\theta}_{T,T}^2]\tilde{c}_\kappa^2+
(1+\sqrt{2})\tilde{\theta}_{T,T}\tilde{c}_\kappa<1$,
where 
$\tilde{b}_1:=1+\tilde{\theta}_{T,T}\tilde{c}_{\kappa}$, $\tilde{b}_2:=\frac{\tilde{\gamma} \tilde{b}_1}{(1-\tilde{\gamma})\tilde{\theta}_{T,T}}+\tilde{c}_\kappa$,
and $h_2(T_j):=\max_{A\subseteq [p]\setminus \{j\}, |A|\le T_j} \| \mb{Z}_{*A}^\top \bd{\epsilon}_j\|_2/n.$

We now consider the bounds for $\max_{j\in[p]}\| \bd{\beta}_j^*\|_2$
and $h_2(T)$.
By  \eqref{lb min Sigma_ii} and \eqref{norm-2 of alpha_j},
we have
\begin{align}
\max_{j\in[p]} \|\bd{\beta}_j^*\|_2&=
\max_{j\in[p]}\|\widehat{\mb{\Gamma}}_{\setminus j, \setminus j}^{1/2}\bd{\alpha}_j^*\|_2
\le \max_{j\in[p]}
\| \widehat{\mb{\Gamma}}_{\setminus j, \setminus j}^{1/2} \|_2\| \bd{\alpha}_j^*\|_2\nonumber\\\
&\le \max_{j\in[p]} \widehat{\mb{\Sigma}}_{jj}^{1/2} (\kappa^2-1)^{1/2}
\le (2\kappa)^{1/2}(\kappa^2-1)^{1/2}
=(2\kappa^3-2\kappa)^{1/2} \label{bounds for |betaj*|_2}
\end{align}
with probability $1-O(p^{-C})$.
By the sub-Gaussianity of $\{x_j,\epsilon_j\}_{j=1}^p$
given in Assumption~\ref{subGauss} and inequality \eqref{subGauss epsilon}, 
$\cov(x_i,\epsilon_j)=0$ for $i\ne j$,
the sub-Gaussian concentration inequality in (29) of \citet{CaiClime}, 
and $(\log p)/n=o(1)$, we have 
\be\label{bound for xe}
P\left(\max_{A\subseteq [p]\setminus \{j\}, |A|\le T_j}\| \mb{X}_{* A}^\top \bd{\epsilon}_j \|_2/n\le
\tilde{M}_2\sqrt{T_j(\log p)/n},\forall j\in[p] \right)=1-O(p^{-C})
\ee
for some constant $\tilde{M}_2>0$ dependent on $C$.
By
$\| \mb{X}_{* A}^\top \bd{\epsilon}_j \|_2
=\| \widehat{\mb{\Gamma}}_{AA}^{1/2}\mb{Z}_{* A}^\top \bd{\epsilon}_j \|_2\ge \min\limits_{i\in [p]} \widehat{\mb{\Sigma}}_{ii}^{1/2}\| \mb{Z}_{* A}^\top \bd{\epsilon}_j \|_2$,
\eqref{lb min Sigma_ii}
and \eqref{bound for xe}, we have
\be\label{bound for ze}
P\left(h_2(T_j)
:=\max_{A\subseteq [p]\setminus \{j\}, |A|\le T_j} \| \mb{Z}_{*A}^\top \bd{\epsilon}_j\|_2/n
\le
 \tilde{M}_2\sqrt{2\kappa T_j(\log p)/n}, 
\forall j\in[p]
 \right)=1-O(p^{-C}).
\ee

Due to \eqref{P(Event1)}, \eqref{error in beta_k+1}, \eqref{bounds for |betaj*|_2}
and \eqref{bound for ze},
we have
%using P(A \cap B)\ge 1-P(A^c)-P(B^c)
\begin{align}
P\Big(\|  \bd{\beta}_j^{k_j+1}-
\bd{\beta}_j^* \|_2
&\le \tilde{b}_1 \tilde{\gamma}_T^{k_j} (2\kappa^3-2\kappa)^{1/2}
+\tilde{b}_2  \tilde{M}_2\sqrt{2\kappa T_j (\log p)/n}
\nonumber\\
&\le [\tilde{b}_1(2\kappa^3-2\kappa)^{1/2}+ \tilde{b}_2  \tilde{M}_2\sqrt{2\kappa}]
\sqrt{T_j(\log p)/n}, ~~\forall j\in[p]\Big)=1-O(p^{-C})
\label{initial bound for beta}
\end{align}
when $0<\tilde{\gamma}_T^{}<1$ and $k_j \ge \log_{\tilde{\gamma}_T^{}} \sqrt{T_j(\log p)/n}$ for all $j\in[p]$.
Note that $\tilde{\gamma}_T^{}>0$ when $\tilde{\theta}_{T,T}>0$.

From \eqref{diff in eigen sample corr},
with probability $1-O(p^{-C})$,
\begin{align}
& \max_{\substack{A,B\in S_{T}: A\cap B=\emptyset}} 
\left| \| \mb{Z}_{*A}^\top\mb{Z}_{*B}/n\|_2  -\| \mb{R}_{AB}\|_2\right|
\le 
 \max_{\substack{A,B\in S_{T}: A\cap B=\emptyset}} 
 \| \mb{Z}_{*A}^\top\mb{Z}_{*B}/n-\mb{R}_{AB}\|_2\nonumber\\
 &\qquad\qquad\qquad \le \max_{A\in S_{2T}}\| \mb{Z}_{*A}^\top\mb{Z}_{*A}/n
-\mb{R}_{AA}\|_2
 \le  \tilde{M}_1 \sqrt{T(\log p)/n}.
 \label{R_AB diff}
\end{align}
Thus, 
$P(\tilde{\theta}_{T,T}\le \theta_{T,T}+\tilde{M}_1 \sqrt{T(\log p)/n})=1-O(p^{-C})$
for some $\tilde{\theta}_{T,T}$
satisfying
$0<\tilde{\theta}_{T,T}\ge  \max_{\substack{A,B\in S_{T}: A\cap B=\emptyset}}\| \mb{Z}_{*A}^\top\mb{Z}_{*B}/n\|_2$.
%Note that $\theta_{T,T}\le (c_\kappa-1)\vee (1-c_\kappa^{-1})$ due to (43) in \citet{Huang18}, and $s\sqrt{(\log p)/n}=o(1)$ due to Assumption~\ref{s2(log p)/n=o(1)}.
Since $T(\log p)/n=o(1)$, 
we have
$P(\tilde{\theta}_{T,T}\le \theta_{T,T}+o(1))=1-O(p^{-C})$.
Combining it with
$\theta_{T,T}\le \| \mb{R}\|_2\le  \kappa_1\kappa_2$ (from 
\eqref{eigen bounds for R})
and $\tilde{c}_\kappa=c_\kappa+o(1)=\kappa_1\kappa_2+o(1)$
yields
$P(\tilde{\gamma}_T^{}\le \gamma_T^{}+o(1))=1-O(p^{-C})$.
Then by $\gamma_T^{}<1-c_1$, 
with probability $1-O(p^{-C})$ we have
the following results:
$\tilde{\gamma}_T^{}\le \gamma_T^{}+c_1/2<1-c_1/2$,  
both $\tilde{b}_1$ and $\tilde{b}_2$ are upper bounded by a constant,
and
 $\log_{1-c_1/2} \sqrt{T_j(\log p)/n}\ge \log_{\tilde{\gamma}_T^{}} \sqrt{T_j(\log p)/n}$ for all $j\in[p]$.
Thus, 
if $k_j\ge \log_{1-c_1/2} \sqrt{T_j(\log p)/n}$ for all $j\in[p]$, 
then by \eqref{initial bound for beta} we have
\be\label{norm-2 error of beta_hat}
P\left(\|  \widehat{\bd{\beta}}_j-
\bd{\beta}_j^* \|_2\le M_3 \sqrt{T_j(\log p)/n},\forall j\in[p]\right)=1-O(p^{-C})
\ee
for some constant $M_3>0$ dependent on $C$.

Note that
\be\label{decomp in L1}
\| \widehat{\mb{\Omega}}-\mb{\Omega}\|_1
\le \max_{j\in[p]} |\widehat{\mb{\Omega}}_{jj}-\mb{\Omega}_{jj}|
+\max_{j\in[p]}\|\widehat{\mb{\Omega}}_{\setminus j,j}-\mb{\Omega}_{\setminus j,j} \|_1.
\ee

We first consider the first term on the right-hand side (RHS) of  \eqref{decomp in L1}.
For all $j\in [p]$,
\begin{align}
\widehat{\sigma}^2_j-\sigma_j^2
&=\|\mb{X}_{*j}-\mb{Z}_{*\setminus j}\widehat{\bd{\beta}}_j\|_2^2/n -\sigma_j^2
= \|\mb{Z}_{*\setminus j}(\bd{\beta}_j^*-\widehat{\bd{\beta}}_j)
+\bd{\epsilon}_j\|_2^2/n -\sigma_j^2\nonumber\\
&=( \| \bd{\epsilon}_j\|^2_2/n-\sigma_j^2  )
+\|\mb{Z}_{*\setminus j}(\bd{\beta}_j^*-\widehat{\bd{\beta}}_j)\|_2^2/n
+2
(\bd{\beta}_j^*-\widehat{\bd{\beta}}_j)^\top\mb{Z}_{*\setminus j}^\top \bd{\epsilon}_j/n.
\label{Theta^-1 decomp}
\end{align}
The bound for the first term on the RHS of \eqref{Theta^-1 decomp} is given in Lemma~\ref{lemma: bound for var(e)}.
For the second term, 
by \eqref{P(Event1)} and \eqref{norm-2 error of beta_hat},
with probability $1-O(p^{-C})$
we have that
\begin{align}
&\max_{j\in[p]}\|\mb{Z}_{*\setminus j}(\widehat{\bd{\beta}}_j-\bd{\beta}_j^*)\|_2^2/n
=\max_{j\in[p]}\|\mb{Z}_{*S_j}[(\widehat{\bd{\beta}}_j)_{S_j}-(\bd{\beta}_j^*)_{S_j}]\|_2^2/n\nonumber\\
&\quad\le\max_{j\in[p]}  n^{-1}\|\mb{Z}_{*S_j}\|_2^2 \|(\widehat{\bd{\beta}}_j)_{S_j}-(\bd{\beta}_j^*)_{S_j}\|_2^2
=\max_{j\in[p]} \|  n^{-1}  \mb{Z}_{*S_j}^\top \mb{Z}_{*S_j}\|_2\|\widehat{\bd{\beta}}_j-\bd{\beta}_j^*\|_2^2\nonumber\\
&\quad\le 2c_\kappa M_3^2 T(\log p)/n
\label{max z diff beta}
\end{align}
where $S_j=A_j^*\cup \widehat{A}_j$.
Now consider the third term on the RHS  of  \eqref{Theta^-1 decomp}.
Similar to \eqref{bound for xe} and \eqref{bound for ze}, we can obtain 
\be\label{max norm for Xe}
P\left(\max_{j\in [p]} \|\mb{X}_{*\setminus j}^\top \bd{\epsilon}_j \|_\infty/n
\le M_4\sqrt{(\log p)/n}\right)=1-O(p^{-C})
\ee
and
\be\label{max norm for Ze}
P\left(\max_{j\in [p]} \|\mb{Z}_{*\setminus j}^\top \bd{\epsilon}_j \|_\infty/n
\le M_4\sqrt{(\log p)/n}\right)=1-O(p^{-C})
\ee
for some constant $M_4>0$ dependent on $C$.
From 
\eqref{norm-2 error of beta_hat} and \eqref{max norm for Ze},
 we have
\begin{align}
&\max_{j\in[p]}2
|(\widehat{\bd{\beta}}_j-\bd{\beta}_j^*)^\top\mb{Z}_{*\setminus j}^\top \bd{\epsilon}_j|/n
\le 2 
(\max_{j\in[p]}\|\widehat{\bd{\beta}}_j-\bd{\beta}_j^* \|_1)
\max_{j\in[p]}\|\mb{Z}_{*\setminus j}^\top \bd{\epsilon}_j \|_\infty/n
\nonumber\\
&\quad \le 2 
(\max_{j\in[p]}\|\widehat{\bd{\beta}}_j-\bd{\beta}_j^* \|_2\sqrt{T_j+s_j-1})
\max_{j\in [p]} \|\mb{Z}_{*\setminus j}^\top \bd{\epsilon}_j \|_\infty/n
\nonumber\\
&\quad \le M_5  T(\log p)/n 
\label{3rd term in Theta^-1 bound}
\end{align}
with probability $1-O(p^{-C})$
for some constant $M_5>0$ dependent on $C$.
Combining \eqref{Theta^-1 decomp}, Lemma~\ref{lemma: bound for var(e)}, \eqref{max z diff beta}, and \eqref{3rd term in Theta^-1 bound} 
yields
\begin{align*}
\max_{j\in[p]} |\widehat{\sigma}^2_j-\sigma_j^2|
&\le M \sqrt{(\log p)/n} + 2c_\kappa M_3^2  T(\log p)/n + M_5  T(\log p)/n\nonumber\\
&\le M_6 [\sqrt{(\log p)/n}+T(\log p)/n]
\end{align*}
with probability $1-O(p^{-C})$ for some constant $M_6>0$ 
dependent on $C$.
From the above inequality, $\sigma_j^2=\mb{\Omega}_{jj}^{-1}$, Assumption~\ref{Theta bound} and $T(\log p)/n=o(1)$,
we have
\[
P((2\kappa)^{-1}\le \min_{j\in[p]}\widehat{\sigma}_j\le \max_{j\in[p]}\widehat{\sigma}_j\le 2\kappa)=1-O(p^{-C}),
\]
which implies
\be\label{Omega exists}
P(\widehat{\mb{\Omega}}_{jj}:=\widehat{\sigma}_j^{-2}~\text{exists}, \forall j\in[p]) =1-O(p^{-C}).
\ee
and
\be\label{bound for Theta_jj}
P((2\kappa)^{-1}\le\min_{j\in[p]}\widehat{\mb{\Omega}}_{jj}\le\max_{j\in[p]}\widehat{\mb{\Omega}}_{jj}\le 2\kappa)=
1-O(p^{-C}).
\ee
Hence, we have
\be\label{diff in Theta[j,j]}
\max_{j\in[p]} |\widehat{\mb{\Omega}}_{jj}-\mb{\Omega}_{jj}|
\le (\max_{j\in[p]} \widehat{\mb{\Omega}}_{jj})(\max_{j\in[p]}\mb{\Omega}_{jj})
	\max_{j\in[p]} |\widehat{\mb{\Omega}}_{jj}^{-1}-\mb{\Omega}_{jj}^{-1}|  
\le 2\kappa^2M_6  [\sqrt{(\log p)/n}+T(\log p)/n]
\ee
with probability $1-O(p^{-C})$.

Next, we consider the bound for the second term on the RHS of \eqref{decomp in L1}.
For any $j\in [p]$,
\begin{align}
\lefteqn{\|\widehat{\mb{\Omega}}_{\setminus j,j}-\mb{\Omega}_{\setminus j,j} \|_1}\nonumber\\
&=\| -\widehat{\mb{\Omega}}_{jj}
\widehat{\mb{\Gamma}}_{\setminus j, \setminus j}^{-1/2} \widehat{\bd{\beta}}_j+\mb{\Omega}_{jj}\widehat{\mb{\Gamma}}_{\setminus j, \setminus j}^{-1/2} \bd{\beta}_j^* \|_1
=\| \widehat{\mb{\Gamma}}_{\setminus j, \setminus j}^{-1/2} (\widehat{\mb{\Omega}}_{jj}\widehat{\bd{\beta}}_j -\mb{\Omega}_{jj}\bd{\beta}_j^*)  \|_1
\nonumber\\
&\le \| \widehat{\mb{\Gamma}}_{\setminus j, \setminus j}^{-1/2} \|_1
\widehat{\mb{\Omega}}_{jj}\| \widehat{\bd{\beta}}_j-\bd{\beta}_j^*\|_1
+|\widehat{\mb{\Omega}}_{jj}-\mb{\Omega}_{jj}|\|    \widehat{\mb{\Gamma}}_{\setminus j,\setminus, j}^{-1/2}\bd{\beta}_j^*\|_1
\nonumber\\
&\le (\max_{k\in[p]} \widehat{\mb{\Sigma}}_{kk}^{-1/2})
\widehat{\mb{\Omega}}_{jj}\| \widehat{\bd{\beta}}_j-\bd{\beta}_j^*\|_1
+|\widehat{\mb{\Omega}}_{jj}-\mb{\Omega}_{jj}|
\|\mb{\Omega}_{\setminus j, j}\mb{\Omega}_{jj}^{-1} \|_1
\label{diff in Theta column}\\
&\le (\min_{k\in[p]}\widehat{\mb{\Sigma}}_{kk})^{-1/2}
\widehat{\mb{\Omega}}_{jj}\| \widehat{\bd{\beta}}_j-\bd{\beta}_j^*\|_2\sqrt{T_j+s_j-1}
+|\widehat{\mb{\Omega}}_{jj}-\mb{\Omega}_{jj}|
\|\mb{\Omega}_{\setminus j,j}\|_1(\min_{k\in[p]}\mb{\Omega}_{kk})^{-1} 
\nonumber\\
&\le (\min_{k\in[p]}\widehat{\mb{\Sigma}}_{kk})^{-1/2}
\widehat{\mb{\Omega}}_{jj}\| \widehat{\bd{\beta}}_j-\bd{\beta}_j^*\|_2\sqrt{2T}
+|\widehat{\mb{\Omega}}_{jj}-\mb{\Omega}_{jj}|
(\|\mb{\Omega}_{\setminus j,j}\|_2\sqrt{T})(\min_{k\in[p]}\mb{\Omega}_{kk})^{-1}.
\nonumber
\end{align}
Together with \eqref{lb min Sigma_ii}, \eqref{bound for Theta_jj}, \eqref{norm-2 error of beta_hat}, \eqref{diff in Theta[j,j]}, $\|\mb{\Omega}_{\setminus j,j}\|_2\le \|\mb{\Omega}\|_2\le \kappa$, 
$\min_k \mb{\Omega}_{kk}\ge \lambda_{\min}(\mb{\Omega})\ge 1/\kappa$,
and $T(\log p)/n=o(1)$,
we have 
\be\label{diff in Theta[-j,j]}
P(\max_{j\in [p]}\|\widehat{\mb{\Omega}}_{\setminus j,j}-\mb{\Omega}_{\setminus j,j} \|_1
\le M_7 T \sqrt{(\log p)/n})=1-O(p^{-C})
\ee
for some constant $M_7>0$ dependent on $C$.

Combining \eqref{decomp in L1}, \eqref{diff in Theta[j,j]},  \eqref{diff in Theta[-j,j]} and $T(\log p)/n=o(1)$ yields 
\be\label{norm-1 bound for Theta}
P\left(\| \widehat{\mb{\Omega}}-\mb{\Omega}\|_1\le M_8 T  \sqrt{(\log p)/n}\right)
	=1-O(p^{-C})
\ee
	with some constant $M_8>0$ dependent on $C$.

Similar to \eqref{diff in Theta column}, from
\eqref{bound for Theta_jj}, \eqref{lb min Sigma_ii}, \eqref{norm-2 error of beta_hat}, \eqref{diff in Theta[j,j]},  \eqref{bounds for |betaj*|_2}, and $T(\log p)/n=o(1)$, we have
\begin{align}\label{column 2-norm bound}
P\Big(\|\widehat{\mb{\Omega}}_{\setminus j,j}-\mb{\Omega}_{\setminus j,j} \|_2
&\le \widehat{\mb{\Omega}}_{jj}\| \widehat{\mb{\Gamma}}_{\setminus j,\setminus j}^{-1/2} \|_2\|\widehat{\bd{\beta}}_j-\bd{\beta}_j^* \|_2
+|\widehat{\mb{\Omega}}_{jj}-\mb{\Omega}_{jj}|\|    \widehat{\mb{\Gamma}}_{\setminus j,\setminus, j}^{-1/2}\|_2\|\bd{\beta}_j^*\|_2\nonumber\\
&\le \tilde{M}_8\sqrt{T_j(\log p )/n},~~\forall j\in[p]\Big)=
1-O(p^{-C})
\end{align}
for some  constant $\tilde{M}_8>0$ dependent on $C$.
Due to \eqref{column 2-norm bound}, \eqref{diff in Theta[j,j]}, 
$\|\widehat{\mb{\Omega}}_{*j}-\mb{\Omega}_{*j} \|_2^2= \|\widehat{\mb{\Omega}}_{\setminus j,j}-\mb{\Omega}_{\setminus j,j} \|_2^2+|\widehat{\mb{\Omega}}_{jj}-\mb{\Omega}_{jj}|^2$,
and $T(\log p)/n=o(1)$, it holds with probability $1-O(p^{-C})$ that
\be\label{Theta column norm 2}
\max\{\| \widehat{\mb{\Omega}}-\mb{\Omega}\|_{\max}, p^{-1/2}\| \widehat{\mb{\Omega}}-\mb{\Omega} \|_F\} \le\max_{j\in[p]}\|\widehat{\mb{\Omega}}_{*j}-\mb{\Omega}_{*j} \|_2\le 
\sqrt{(\tilde{M}_8^2+ 4\kappa^4M_6^2)T(\log p)/n}.
\ee

Next, we consider  the error bounds of $\widehat{\mb{\Omega}}$ under Assumption~\ref{(C2)}.

By Theorem~12\,(i) in \citet{Huang18}, 
for
any $\tilde{\mu}\ge \max_{i\ne j}|\widehat{\mb{R}}_{ij}|$
with $\widehat{\mb{R}}:=\mb{Z}^\top\mb{Z}/n$,
when $T\tilde{\mu}\le 1/4$, 
we have that, for all $j\in[p]$,
\begin{align}
\| \bd{\beta}_j^{k_j+1}-\bd{\beta}_j^* \|_{\infty}
&\le \frac{4}{3}\tilde{\gamma}_{T,\mu}^{k_j} \|\bd{\beta}_j^* \|_\infty
+\frac{4}{3}(\frac{4}{1-\tilde{\gamma}_{T,\mu}^{}}+1)h_\infty(T),
\label{err in beta_k+1 in infty norm}\\
\| \bd{\beta}_j^*|_{A_j^*\setminus A_j^{k_j}}\|_\infty 
&\le \tilde{\gamma}_{T,\mu}^{k_j} \| \bd{\beta}_j^*  \|_\infty
+\frac{4}{1-\tilde{\gamma}_{T,\mu}^{}}h_\infty(T),
\label{infty norm beta outside}
\end{align}
where 
$
\tilde{\gamma}_{T,\mu}^{}:=(1+2T\tilde{\mu})T\tilde{\mu}/[1-(T-1)\tilde{\mu}]+2T\tilde{\mu}<1
$ and
$h_\infty(T):=\max\limits_{{j\in[p]}, A\subseteq [p]\setminus \{j\}, |A|\le T}\| \mb{Z}_{*A}^\top\bd{\epsilon}_j \|_\infty/n$.
By \eqref{err in beta_k+1 in infty norm}, $\| \bd{\beta}_j^* \|_\infty\le \| \bd{\beta}_j^* \|_2$,
\eqref{bounds for |betaj*|_2} and \eqref{max norm for Ze},
we have that, 
when $T\tilde{\mu}\le 1/4$,
with probability $1-O(p^{-C})$, 
\begin{align}
\max_{j\in[p]}\| \bd{\beta}_j^{k_j+1}-\bd{\beta}_j^* \|_{\infty}
&\le
\frac{4}{3}(2\kappa^3-2\kappa)^{1/2}\max_{j\in[p]}\tilde{\gamma}_{T,\mu}^{k_j}
+\frac{4}{3}(\frac{4}{1-\tilde{\gamma}_{T,\mu}^{}}+1)
M_4\sqrt{(\log p)/n}\nonumber\\
&\le [\frac{4}{3}(2\kappa^3-2\kappa)^{1/2}
+\frac{4}{3}(\frac{4}{1-\tilde{\gamma}_{T,\mu}}^{}+1)
M_4]\sqrt{(\log p)/n}\label{initial infty bound for beta}
\end{align}
if $0<\tilde{\gamma}_{T,\mu}^{}<1$ and $\min_{j\in[p]}k_j\ge \log_{\tilde{\gamma}_{T,\mu}^{}}\sqrt{(\log p)/n}$.
Note that $\tilde{\gamma}_{T,\mu}^{}>0$ when $\tilde{\mu}>0$.
From  \eqref{lb for sigma_ii} and \eqref{max norm err of Sigma}, 
\[
 \| \mb{\Gamma}^{-1/2} (\widehat{\mb{\Sigma}} - \mb{\Sigma}) \mb{\Gamma}^{-1/2}  \|_{\max}
 \le (\min_{i\in[p]} \mb{\Sigma}_{ii})^{-1}\| \widehat{\mb{\Sigma}} - \mb{\Sigma}\|_{\max}
\le \kappa_2\tilde{M}_0 \sqrt{(\log p)/n}
\]
with probability $1-O(p^{-C})$.
Using the above inequality and
following the proof of (S.15) in the supplement of \cite{Shu19} yields
\be\label{max err of sample corr}
P\left(\|\mb{Z}^\top\mb{Z}/n - \mb{R} \|_{\max}\le 4  \kappa_2\tilde{M}_0 \sqrt{(\log p)/n}\right)=1-O(p^{-C}).
\ee
Thus, we have
$
P(\tilde{\mu}\le \mu+ 4\kappa_2\tilde{M}_0\sqrt{(\log p)/n})=1-O(p^{-C})$
for some $\tilde{\mu}$ satisfying $0<\tilde{\mu}\ge \max_{i\ne j}|\widehat{\mb{R}}_{ij}|$.
Then by $T\mu< 1/4-c_2$ and $T\sqrt{(\log p)/n}=o(1)$, 
with probability $1-O(p^{-C})$
we have the following results:
$T\tilde{\mu}\le T\mu+4\kappa_2\tilde{M}_0T\sqrt{(\log p)/n}
\le T\mu+c_2/2<1/4-c_2/2
$,
$\tilde{\gamma}_{T,\mu}^{}< [1+2(1/4-c_2/2)](1/4-c_2/2)/[1-(1/4-c_2/2)]+2(1/4-c_2/2)=(3-6c_2)/(3+2c_2)<1$,
and
$\log_{\frac{3-6c_2}{3+2c_2}}\sqrt{(\log p)/n}\ge \log_{\tilde{\gamma}_{T,\mu}^{}}\sqrt{(\log p)/n}$.
Thus, if $\min_{j\in[p]} k_j\ge \log_{\frac{3-6c_2}{3+2c_2}}\sqrt{(\log p)/n}$,
then by \eqref{initial infty bound for beta}
we obtain
\be\label{infty bound for beta}
\max_{j\in[p]}\| \widehat{\bd{\beta}}_j-\bd{\beta}_j^*\|_{\infty}\le M_9\sqrt{(\log p)/n}
\ee
and
\be\label{norm-2 error of beta_hat, 2}
\max_{j\in[p]}\|\widehat{\bd{\beta}}_j-\bd{\beta}_j^*\|_2\le \sqrt{2T}\max_{j\in[p]}\|\widehat{\bd{\beta}}_j-\bd{\beta}_j^*\|_\infty
\le M_9\sqrt{2T(\log p)/n}
\ee
with probability $1-O(p^{-C})$ for some constant $M_9>0$ dependent on $C$.

Using \eqref{norm-2 error of beta_hat, 2} instead of \eqref{norm-2 error of beta_hat},
we can still obtain
all the results given in the part from \eqref{decomp in L1}
to \eqref{Theta column norm 2},
including \eqref{norm-1 bound for Theta} and \eqref{Theta column norm 2}.
But we can refine the error bound in the max norm
using \eqref{infty bound for beta}.
Specifically, 
note that
\be\label{decomp max norm}
\| \widehat{\mb{\Omega}}-\mb{\Omega}\|_{\max}
\le 
\max_{j\in[p]}\|\widehat{\mb{\Omega}}_{\setminus j, j}-\mb{\Omega}_{\setminus j, j}\|_{\infty}+\max_{j\in[p]}| \widehat{\mb{\Omega}}_{jj}-\mb{\Omega}_{jj}|.
\ee
Similar to \eqref{diff in Theta column} and \eqref{diff in Theta[-j,j]}, 
by \eqref{infty bound for beta}, \eqref{diff in Theta[j,j]}
and $T\sqrt{(\log p)/n}=o(1)$, 
the first term on the RHS of~\eqref{decomp max norm}
\begin{align}
\lefteqn{\max_{j\in[p]}\|\widehat{\mb{\Omega}}_{\setminus j, j}-\mb{\Omega}_{\setminus j, j}\|_{\infty}}\nonumber\\
&\le
\max_{j\in[p]}\left\{
\max_{k\in[p]}(\widehat{\mb{\Sigma}}_{kk}^{-1/2})
\widehat{\mb{\Omega}}_{jj}\| \widehat{\bd{\beta}}_j-\bd{\beta}_j^*\|_{\infty}
+|\widehat{\mb{\Omega}}_{jj}-\mb{\Omega}_{jj}|
\|\mb{\Omega}_{\setminus j, j} \|_{\infty}\mb{\Omega}_{jj}^{-1}
\right\}\nonumber\\
&\le (2\kappa)^{3/2}
M_9\sqrt{(\log p)/n}
+2\kappa^4M_6  [\sqrt{(\log p)/n}+T(\log p)/n]\nonumber\\
&\le [(2\kappa)^{3/2}+4\kappa^4M_6 ]\sqrt{(\log p)/n}
\label{max infty norm for diff in Theta column}
\end{align}
with probability $1-O(p^{-C})$.
Combining \eqref{decomp max norm},  \eqref{max infty norm for diff in Theta column} and
\eqref{diff in Theta[j,j]} 
gives
\be\label{max norm bound for diff in Theta}
P\left(\| \widehat{\mb{\Omega}}-\mb{\Omega}\|_{\max}\le [(2\kappa)^{3/2}+8\kappa^4M_6] \sqrt{(\log p)/n}\right)=1-O(p^{-C}).
\ee

It is easily seen from the proof that the above results hold uniformly for  $\mb{\Omega}\in\mathcal{G}(T+1)$. 

The proof is complete.
\end{proof}	

\begin{proof}[Proof of Theorem~\ref{thm: Sparsistency} and Lemma~\ref{Lemma: A_hat=A}]

We first consider the diagonal entries of $\widehat{\mb{\Omega}}$.
By Assumption~\ref{Theta bound}, we have
$\min_{j\in[p]}\mb{\Omega}_{jj}\ge \lambda_{\min}(\mb{\Omega})\ge \kappa_1^{-1}>0$.
Recall from the proof of Theorem~\ref{thm: consistency}
that \eqref{bound for Theta_jj} holds
under \ref{(C1)} or \ref{(C2)}.
Thus,
 $P(\min_{j\in[p]}\widehat{\mb{\Omega}}_{jj}\ge (2\kappa)^{-1} > 0)=1-O(p^{-C})$.

Now consider the off-diagonal entries of $\widehat{\mb{\Omega}}$.
Under Assumption~\ref{(C1)}, from 
\eqref{column 2-norm bound} we have
\[
P\left(	
\| \widehat{\mb{\Omega}}_{\setminus j,j}-\mb{\Omega}_{\setminus j,j}\|_{\max}\le 
\tilde{M}_8\sqrt{T_j(\log p)/n},~~\forall j\in[p]
 \right)
=1-O(p^{-C}).
\]
Thus, when $m_j\ge M_1\sqrt{T_j(\log p)/n} $
for all $j\in [p]$
with a constant $M_1>\tilde{M}_8$,
it holds with probability $1-O(p^{-C})$
that
$|\widehat{\mb{\Omega}}_{ij}|>0$
and $\sign(\widehat{\mb{\Omega}}_{ij})=\sign(\mb{\Omega}_{ij})$
for $i\ne j$ and $(i,j)\in \supp(\mb{\Omega})$.
Similarly,
under Assumption~\ref{(C2)} with
$\min_{j\in[p]}m_j\ge M_2\sqrt{(\log p)/n}$
for some  constant $M_2>0$,
by using
\eqref{max norm bound for diff in Theta}
and letting
$M_2>(2\kappa)^{3/2}+8\kappa^4M_6$,
we obtain the same result
for support and sign recovery.

Note that the above results hold uniformly for  $\mb{\Omega}\in\mathcal{G}(T+1)$. Thus,
under Assumption~\ref{(C1)} with $m_j\ge M_1\sqrt{T_j(\log p)/n} $ 
for all $j\in[p]$ and $M_1>\tilde{M}_8$, or under Assumption~\ref{(C2)} with
$\min_{j\in[p]}m_j\ge M_2\sqrt{(\log p)/n}$ 
and $M_2>(2\kappa)^{3/2}+8\kappa^4M_6$, we have
\be\label{support contain}
\inf_{\mb{\Omega}\in\mathcal{G}(T+1)}P(\supp(\mb{\Omega})\subseteq \supp(\widehat{\mb{\Omega}}))
=1-O(p^{-C})
\ee
and
$$\inf_{\mb{\Omega}\in\mathcal{G}(T+1)}P(\sign(\widehat{\mb{\Omega}}_{ij})=\sign(\mb{\Omega}_{ij}),
\forall (i,j)\in \supp(\mb{\Omega}))=1-O(p^{-C}).$$

Recall that 
$\mb{\Omega}_{\setminus j,j}=-\mb{\Omega}_{jj}\bd{\alpha}_j^*
=-\mb{\Omega}_{jj}\widehat{\mb{\Gamma}}_{\setminus j, \setminus j}^{-1/2}\bd{\beta}_j^*$
and
$\widehat{\mb{\Omega}}_{\setminus j,j}=-\widehat{\mb{\Omega}}_{jj}\widehat{\bd{\alpha}}_j
=-\widehat{\mb{\Omega}}_{jj}\widehat{\mb{\Gamma}}_{\setminus j, \setminus j}^{-1/2}\widehat{\bd{\beta}}_j$.
From \eqref{Gamma exists} and 
$\min_{j\in[p]}\mb{\Omega}_{jj}\ge\kappa_1^{-1}>0$, we have
\be\label{supp(Omega)=A}
\inf_{\mb{\Omega}\in\mathcal{G}(T+1)} P(\supp(\mb{\Omega}_{\setminus j,j})=A_j^*:=\supp(\bd{\alpha}^*_j)=\supp(\bd{\beta}^*_j),\forall j\in [p])=1-O(p^{-C}).
\ee
It follows from \eqref{Gamma exists}, \eqref{Omega exists} and \eqref{bound for Theta_jj} that
\be\label{supp(Omega_hat)=A_hat}
\inf_{\mb{\Omega}\in\mathcal{G}(T+1)} P(\supp(\widehat{\mb{\Omega}}_{\setminus j,j})=\supp(\widehat{\bd{\beta}}_j)\subseteq \widehat{A}_j,\forall j\in[p])=1-O(p^{-C}).
\ee
Combining \eqref{support contain}, \eqref{supp(Omega)=A} and \eqref{supp(Omega_hat)=A_hat} yields
\[
\inf_{\mb{\Omega}\in\mathcal{G}(T+1)}P(A_j^*\subseteq\widehat{A}_j, \forall j\in [p])=1-O(p^{-C}).
\]

Next, we consider the second part of the theorem and that of the lemma.

We first give the proof under Assumption~\ref{(C1)}.
Combining \eqref{P(Event1)}, \eqref{norm of beta outside}, \eqref{bounds for |betaj*|_2}
and \eqref{bound for ze},
we obtain 
\begin{align*}
P\Big(&\| (\bd{\beta}_j^*)_{A_j^*\setminus A_j^{k_j}}\|_2 
\le \tilde{\gamma}_T^{k_j}(2\kappa^3-2\kappa)^{1/2}
+\frac{\tilde{\gamma}_T^{}}{(1-\tilde{\gamma}_T^{})\tilde{\theta}_{T,T}}\tilde{M}_2\sqrt{2\kappa T_j(\log p)/n}\\
&\qquad\le [(2\kappa^3-2\kappa)^{1/2}+\frac{\tilde{\gamma}_T^{}}{(1-\tilde{\gamma}_T^{})\tilde{\theta}_{T,T}}\tilde{M}_2\sqrt{2\kappa}]\sqrt{T_j(\log p)/n},~~\forall j\in[p]\Big)=1-O(p^{-C})
\end{align*}
when $0<\tilde{\gamma}_T^{}<1$ and $k_j \ge \log_{\tilde{\gamma}_T^{}} \sqrt{T(\log p)/n}$.
By the argument below \eqref{R_AB diff},
when $\gamma_T^{}<1-c_1$,
with probability $1-O(p^{-C})$
we have the following results:
$\tilde{\gamma}_T^{}<1-c_1/2$,
$
\frac{\tilde{\gamma}_T^{}}{(1-\tilde{\gamma}_T^{})\tilde{\theta}_{T,T}}\le M_0
$
for some constant $M_0>0$,
and  $\log_{1-c_1/2} \sqrt{T_j(\log p)/n}\ge \log_{\tilde{\gamma}_T^{}} \sqrt{T_j(\log p)/n}$ for all $j\in[p]$.
Thus, when $k_j\ge \log_{1-c_1/2} \sqrt{T_j(\log p)/n}$ for all $j\in[p]$,
with probability $1-O(p^{-C})$ we have
\[
P(\| (\bd{\beta}_j^*)_{A_j^*\setminus A_j^{k_j}}\|_2
\le [(2\kappa^3-2\kappa)^{1/2}+M_0\tilde{M}_2\sqrt{2\kappa }]\sqrt{T_j(\log p)/n},~\forall j\in[p])=1-O(p^{-C}).
\]
Combining it with 
\eqref{Gamma exists},
\eqref{lb min Sigma_ii} and $\|\mb{\Omega}\|_{\max}\le\lambda_{\max}(\mb{\Omega}) \le \kappa $ yields that
the event
\begin{align}
\Big\{\|(\mb{\Omega}_{\setminus j,j})_{A_j^*\setminus A_j^{k_j}}\|_2
&=
\|-\mb{\Omega}_{jj}(\bd{\alpha}_j^*)_{A_j^*\setminus A_j^{k_j}}\|_2
=
\|-\mb{\Omega}_{jj}(\widehat{\mb{\Gamma}}_{\setminus j, \setminus j}^{-1/2}\bd{\beta}_j^*)_{A_j^*\setminus A_j^{k_j}}\|_2\nonumber\\
&\le \mb{\Omega}_{jj} (\min_{j\in[p]} \widehat{\mb{\Sigma}}_{ii})^{-1/2}\| (\bd{\beta}_j^*)_{A_j^*\setminus A_j^{k_j}}\|_2\nonumber\\
&\le \kappa (2\kappa)^{1/2}[(2\kappa^3-2\kappa)^{1/2}+M_0\tilde{M}_2\sqrt{2\kappa}]\sqrt{T_j(\log p)/n}\nonumber\\
&<M_1 \sqrt{T_j(\log p)/n}\le m_j,
~~\forall j\in[p]\Big\}
\label{||Theta||_2<m_theta}
\end{align}
holds with probability $1-O(p^{-C})$,
where we let the constant $M_1>\kappa (2\kappa)^{1/2}[(2\kappa^3-2\kappa)^{1/2}+M_0\tilde{M}_2\sqrt{2\kappa}]$.
Thus, when $|A_j^{k_j}|=s_j-1$ for all $j\in[p]$,
\be\label{A=A_hat}
\inf_{\mb{\Omega}\in\mathcal{G}(T+1)} P(\widehat{A}_j:=A_j^{k_j}=A_j^*, \forall j\in [p])=1-O(p^{-C}).
\ee
From \eqref{support contain}, \eqref{supp(Omega)=A}, \eqref{supp(Omega_hat)=A_hat}, and \eqref{A=A_hat} (note that the last three results hold uniformly for $\mb{\Omega}\in\mathcal{G}(T+1)$),
when $m_j\ge M_1\sqrt{T_j(\log p)/n} $ 
for all $j\in[p]$ with $M_1>\max\{\tilde{M}_8,\kappa (2\kappa)^{1/2}[(2\kappa^3-2\kappa)^{1/2}+M_0\tilde{M}_2\sqrt{2\kappa}]\}$, we have
\[
\inf_{\mb{\Omega}\in\mathcal{G}(T+1)}P(\supp(\mb{\Omega})= \supp(\widehat{\mb{\Omega}}))
=1-O(p^{-C}).
\]

Now we give the proof under Assumption~\ref{(C2)}.
By \eqref{infty norm beta outside}, $\| \bd{\beta}_j^* \|_\infty\le \| \bd{\beta}_j^* \|_2$,
\eqref{bounds for |betaj*|_2} and \eqref{max norm for Ze},
when $T\tilde{\mu}\le 1/4$,
with probability $1-O(p^{-C})$
we have
\begin{align*}
\max_{j\in[p]}\| (\bd{\beta}_j^*)_{A_j^*\setminus A_j^{k_j}}\|_\infty 
&\le(2\kappa^3-2\kappa)^{1/2}\max_{j\in[p]} \tilde{\gamma}_{T,\mu}^{k_j} 
+\frac{4}{1-\tilde{\gamma}_{T,\mu}^{}}M_4\sqrt{(\log p)/n}\\
&\le [(2\kappa^3-2\kappa)^{1/2}+\frac{4}{1-\tilde{\gamma}_{T,\mu}^{}}M_4]\sqrt{(\log p)/n}
\end{align*}
if $0<\tilde{\gamma}_{T,\mu}^{}<1$ and $\min_{j\in[p]}k_j\ge \log_{\tilde{\gamma}_{T,\mu}^{}}\sqrt{(\log p)/n}$.
By the argument below \eqref{max err of sample corr},
when $T\mu< 1/4-c_2$, with probability $1-O(p^{-C})$
we have the following results:
$T\tilde{\mu}<1/4-c_2/2$, 
$\tilde{\gamma}_{T,\mu}^{}<(3-6c_2)/(3+2c_2)<1$,
and $\log_{\frac{3-6c_2}{3+2c_2}}\sqrt{(\log p)/n}\ge \log_{\tilde{\gamma}_{T,\mu}^{}}\sqrt{(\log p)/n}$.
Thus, when $\min_{j\in[p]}k_j\ge \log_{\frac{3-6c_2}{3+2c_2}}\sqrt{(\log p)/n}$,
similar to \eqref{||Theta||_2<m_theta},
with probability $1-O(p^{-C})$
we have
\begin{align*}
\max_{j\in[p]}\|(\mb{\Omega}_{\setminus j,j})_{A_j^*\setminus A_j^{k_j}}\|_\infty
&\le \max_{j\in[p]} |\mb{\Omega}_{jj}| (\min_{j\in[p]} \widehat{\mb{\Sigma}}_{ii})^{-1/2}\max_{j\in[p]}\| (\bd{\beta}_j^*)_{A_j^*\setminus A_j^{k_j}}\|_\infty\\
&\le \kappa(2\kappa)^{1/2}
[(2\kappa^3-2\kappa)^{1/2}+\frac{4}{1-\frac{3-6c_2}{3+2c_2}}M_4]\sqrt{(\log p)/n}\\
&<M_2\sqrt{(\log p)/n}\le \min_{j\in[p]}m_j
\end{align*}
where we let the constant $M_2> \kappa(2\kappa)^{1/2}
[(2\kappa^3-2\kappa)^{1/2}+\frac{4}{1-\frac{3-6c_2}{3+2c_2}}M_4]$.
Thus, 
when $|A_j^{k_j}|=s_j-1$ for all $j\in[p]$, we have
\be\label{A=A_hat2}
\inf_{\mb{\Omega}\in\mathcal{G}(T+1)} P(\widehat{A}_j:=A_j^{k_j}=A_j^*, \forall j\in [p])=1-O(p^{-C}).
\ee
From 
\eqref{support contain}, \eqref{supp(Omega)=A}, \eqref{supp(Omega_hat)=A_hat}
and \eqref{A=A_hat2} (note that the last three results hold uniformly for $\mb{\Omega}\in\mathcal{G}(T+1)$),
when 
$\min_{j\in[p]}m_j\ge 
M_2[\sqrt{(\log p)/n}+s(\log p)/n]$
with $M_2>\max\{(2\kappa)^{3/2}+8\kappa^4M_6,\kappa(2\kappa)^{1/2}
[(2\kappa^3-2\kappa)^{1/2}+\frac{4}{1-\frac{3-6c_2}{3+2c_2}}M_4]\}$,
we have 
\[
\inf_{\mb{\Omega}\in\mathcal{G}(T+1)}P(\supp(\mb{\Omega})= \supp(\widehat{\mb{\Omega}}))
=1-O(p^{-C}).
\]

The proof is complete.
\end{proof}

\begin{proof}[Proof of Lemma~\ref{Lemma2: A_hat=A}]

The proof follows the proof of Theorem~4 given in the appendix of \citet{zhu2020polynomial} with some modifications.
Note that the inequality $P(\hat{\mathcal{A}}^s\supseteq \mathcal{A}^*)\ge 1-O(p^{-\alpha})$ below their (13) can be obtained by our
Lemma~\ref{Lemma: A_hat=A}
when $T_j\ge \max(s_j-1,1)$,
$\gamma_{T_{\max}}^{}<1-c_1$ for some constant $c_1>0$, 
$k_j\ge  \log_{1-\frac{c_1}{2}} \sqrt{T_j(\log p)/n}$, $T_{\max}(\log p)/n=o(1)$ and
$m_j\ge M_1\sqrt{T_{\max,j}(\log p)/n}$.
We revise their (16) to be
$
\mathcal{L}_{\hat{\mathcal{A}}^s}
\ge \kappa_2^{-1}/4,
$
which 
is obtained by our Lemma~\ref{lemma: bound for var(e)}, \eqref{max z diff beta} and $\sigma_j^2=\mb{\Omega}_{jj}^{-1}\ge\kappa_2^{-1}$.
Note that $\sigma^2$ in their (17) and the inequality below their (19) should be corrected to  
$\bd{\epsilon}^\top \bd{\epsilon}/n$,
$\|\bd{\epsilon}^\top \mb{H}_{\mathcal{A}^*}\bd{\epsilon}  \|_2^2$
in the inequality below their (17) and also in their (18)
 should be corrected to $\|\bd{\epsilon}^\top \mb{H}_{\mathcal{A}^*}\bd{\epsilon}  \|_2$, and
 $(c_-(s),c_+(s),\theta_{s,s}^2)$ in the proof from their (17) to the inequality below their (19) should be corrected to  $(c_-(s^*),c_+(s^*),\theta_{s^*,s^*}^2)$.
 From the argument below \eqref{R_AB diff}, 
$P(\tilde{\gamma}_{T_{\max}}^{}<1,
\tilde{\theta}_{T_{\max},T_{\max}}< 2\kappa_1\kappa_2)=1-O(p^{-C})$.
 Let their $c_{-}(s^*)=1/c_{+}(s^*)$ and let it equal our $\tilde{c}_\kappa^{-1}$ which falls in $[(2\kappa_1\kappa_2)^{-1} ,(\kappa_1\kappa_2)^{-1}]$ with probability $1-O(p^{-C})$,
 and let their $\theta_{s^*,s^*}$ be our $\tilde{\theta}_{s^*,s^*}$.
Then, with probability $1-O(p^{-C})$ we have
 their $(1+\sqrt{2})\theta_{s^*,s^*}^2/c_{-}^2(s^*)\le$ our $\tilde{\gamma}_{T_{\max}}^{}<1$.
 Thus, it holds with probability $1-O(p^{-C})$ that
 $(\kappa_1\kappa_2)^{-2}\ge c_{-}^2(s^*)\ge c_{-}^2(s^*)-\theta_{s^*,s^*}^2\ge [1-1/(1+\sqrt{2})]c_{-}^2(s^*)\ge  [1-1/(1+\sqrt{2})](2\kappa_1\kappa_2)^{-2}$.
 Thus, 
with probability $1-O(p^{-C})$, we have that
 their $K_{s,3}, K_{s,4}$ and $K_{s,5}$ with $s=s^*$ are upper and lower bounded by constants. 
Thus, we need their $b^*=M_b\max\{s^* (\log p)/n,(\log p)\log(\log n)/n\}$ for a sufficiently large constant $M_b>0$,
and $s^*(\log p)\log(\log n)/n=o(1)$
to obtain the last upper bound in their (18) and the last inequality in the proof of their Theorem~4.
Translating these conditions into our context,
we need
$
\min_{i\in A_j^*}|(\bd{\beta}_j^*)_{\{i\}}|
= \min_{i\in A_j^*}|(\mb{\Omega}_{jj}^{-1}
\widehat{\mb{\Gamma}}_{\setminus j,\setminus j}^{1/2}\mb{\Omega}_{\setminus j,j})_{\{i\}}|
\ge [M_b \max\{s_j (\log p)/n, (\log p)\log(\log n)/n\}]^{1/2}
$
with probability $1-O(p^{-C})$,
which is satisfied when
$m_j:= \min_{i\in A_j^*}| (\mb{\Omega}_{\setminus j,j})_{\{i\}} |
\ge (2\kappa)^{1/2}\kappa [M_b \max\{s_j (\log p)/n, (\log p)\log(\log n)/n\}]^{1/2}$,
and we also need $s_j(\log p)\log (\log n)/n=o(1)$.
The two conditions are assumed in the lemma.
Consequently, we have $P(T_j^*=s_j-1)=1-O(p^{-C})$
for any given constant $C>0$,
and thus $P(T_j^*=s_j-1, \forall j\in [p])=1-O(p^{-C+1})$,
where we can further let $C>1$. 
Then from Lemma~\ref{Lemma: A_hat=A},
we complete the proof. 
\end{proof}

\begin{proof}[Proof of Corollary~\ref{cor: Sparsistency}]
Combining 
\eqref{eq: T*=s-1} in Lemma~\ref{Lemma2: A_hat=A}
and 
\eqref{eq: supp_hat=supp} in Theorem~\ref{thm: Sparsistency},
we obtain the corollary.
\end{proof}

We now provide the proof of Theorem~\ref{thm: normality 2}, followed by the proof of Lemma~\ref{thm: normality}.

\begin{proof}[Proof of Theorem~\ref{thm: normality 2}]
From \eqref{P(Event1)},
$P((\mb{Z}_{*\widehat{A}_j}^\top \mb{Z}_{*\widehat{A}_j})^{-1}~\text{exists},\forall j\in[p])=1-O(p^{-C})$. 
By (7) in \citet{Huang18}, we have
$	
	(\widehat{\bd{\beta}}_j)_{\widehat{A}_j}=(\mb{Z}_{*\widehat{A}_j}^\top \mb{Z}_{*\widehat{A}_j})^{-1} \mb{Z}_{*\widehat{A}_j}^\top \mb{X}_{*j}.
$
From the above two equations, \eqref{Gamma exists}, and \eqref{Omega exists}, 
it holds with probability $1-O(p^{-C})$ that
\begin{align}
\widehat{\mb{\Omega}}_{\widehat{A}_j, j}&=
-\widehat{\mb{\Omega}}_{jj}(\widehat{\bd{\alpha}}_j)_{\widehat{A}_j}\label{hat: ThetaAj=-Thetajj alphajAj}\\
&=- \widehat{\mb{\Omega}}_{jj}\widehat{\mb{\Gamma}}_{\widehat{A}_j\widehat{A}_j}^{-1/2}  (\widehat{\bd{\beta}}_j)_{\widehat{A}_j}
=- \widehat{\mb{\Omega}}_{jj}\widehat{\mb{\Gamma}}_{\widehat{A}_j\widehat{A}_j}^{-1/2} 
(\mb{Z}_{*\widehat{A}_j}^\top \mb{Z}_{* \widehat{A}_j})^{-1} \mb{Z}_{*\widehat{A}_j}^\top \mb{X}_{*j}\nonumber\\
&=-\widehat{\mb{\Omega}}_{jj}
\widehat{\mb{\Gamma}}_{\widehat{A}_j\widehat{A}_j}^{-1/2} 
[\widehat{\mb{\Gamma}}_{\widehat{A}_j\widehat{A}_j}^{-1/2}(\widehat{\mb{\Gamma}}_{\widehat{A}_j\widehat{A}_j}^{1/2}\mb{Z}_{*\widehat{A}_j}^\top \mb{Z}_{* \widehat{A}_j}\widehat{\mb{\Gamma}}_{\widehat{A}_j\widehat{A}_j}^{1/2})\widehat{\mb{\Gamma}}_{\widehat{A}_j\widehat{A}_j}^{-1/2}]^{-1} \mb{Z}_{*\widehat{A}_j}^\top 
\mb{X}_{*j}\nonumber\\
&=-\widehat{\mb{\Omega}}_{jj}(\mb{X}_{*\widehat{A}_j}^\top \mb{X}_{* \widehat{A}_j})^{-1} \mb{X}_{*\widehat{A}_j}^\top \mb{X}_{*j},~~~\forall j\in[p].
\label{hat: ThetaAj=-Thetajj XXXX}
%\\
%&=-\widehat{\mb{\Omega}}_{jj}\widehat{\mb{\Sigma}}_{\widehat{A}_j,\widehat{A}_j}^{-1}
%\widehat{\mb{\Sigma}}_{\widehat{A}_j,j} \label{hat: ThetaAj=-Thetajj SigmaAA SigmaAj}
\end{align}
From \eqref{hat: ThetaAj=-Thetajj alphajAj} and \eqref{hat: ThetaAj=-Thetajj XXXX},
\be\label{alpha=Sigma^-1 Sigma}
P\left((\widehat{\bd{\alpha}}_j)_{\widehat{A}_j}=\widehat{\mb{\Sigma}}_{\widehat{A}_j,\widehat{A}_j}^{-1}
\widehat{\mb{\Sigma}}_{\widehat{A}_j,j},\forall j\in[p]\right)=1-O(p^{-C}).
\ee
Then, %Using it and \eqref{Aj in Aj_hat} yields that
with probability $1-O(p^{-C})$, we have that
\begin{align}
\widehat{\mb{\Omega}}_{jj}^{-1}&=\widehat{\sigma}_j^2=\frac{1}{n}\big\|\mb{X}_{*j}-\mb{X}_{*\setminus j} \widehat{\bd{\alpha}}_j \big\|_2^2
=\frac{1}{n}\big\|\mb{X}_{*j}-\mb{X}_{*\widehat{A}_j} (\widehat{\bd{\alpha}}_j)_{\widehat{A}_j}  \big\|_2^2\nonumber\\
&=\frac{1}{n}\mb{X}_{*j}^\top \mb{X}_{*j}
-\frac{2}{n}\mb{X}_{*j}^\top\mb{X}_{*\widehat{A}_j} (\widehat{\bd{\alpha}}_j)_{\widehat{A}_j}
+\frac{1}{n}    (\widehat{\bd{\alpha}}_j)_{\widehat{A}_j}^\top   \mb{X}_{*\widehat{A}_j}^\top   \mb{X}_{*\widehat{A}_j} (\widehat{\bd{\alpha}}_j)_{\widehat{A}_j}\nonumber\\
&=\widehat{\mb{\Sigma}}_{jj}-2\widehat{\mb{\Sigma}}_{j,\widehat{A}_j}\widehat{\mb{\Sigma}}_{\widehat{A}_j,\widehat{A}_j}^{-1}
\widehat{\mb{\Sigma}}_{\widehat{A}_j,j}
+ \widehat{\mb{\Sigma}}_{\widehat{A}_j,j}^\top\widehat{\mb{\Sigma}}_{\widehat{A}_j,\widehat{A}_j}^{-1} \widehat{\mb{\Sigma}}_{\widehat{A}_j,\widehat{A}_j}\widehat{\mb{\Sigma}}_{\widehat{A}_j,\widehat{A}_j}^{-1}
\widehat{\mb{\Sigma}}_{\widehat{A}_j,j}\nonumber\\
&=\widehat{\mb{\Sigma}}_{jj}-\widehat{\mb{\Sigma}}_{j,\widehat{A}_j}\widehat{\mb{\Sigma}}_{\widehat{A}_j,\widehat{A}_j}^{-1}
\widehat{\mb{\Sigma}}_{\widehat{A}_j,j},~~~\forall j\in[p].
\label{Theta_jj_hat}
\end{align}
Multiplying both sides of the above equation by $\widehat{\mb{\Omega}}_{jj}$ 
and then using \eqref{hat: ThetaAj=-Thetajj XXXX} yields
\begin{align}
1&=\widehat{\mb{\Sigma}}_{jj}\widehat{\mb{\Omega}}_{jj}-\widehat{\mb{\Sigma}}_{j,\widehat{A}_j}(\widehat{\mb{\Sigma}}_{\widehat{A}_j,\widehat{A}_j}^{-1}
\widehat{\mb{\Sigma}}_{\widehat{A}_j,j}\widehat{\mb{\Omega}}_{jj})\nonumber\\
&=\widehat{\mb{\Sigma}}_{jj}\widehat{\mb{\Omega}}_{jj}+\widehat{\mb{\Sigma}}_{j,\widehat{A}_j}\widehat{\mb{\Omega}}_{\widehat{A}_j,j}.\label{Sigma Theta=1}
\end{align}
On the other hand, 
left multiplying both sides of \eqref{hat: ThetaAj=-Thetajj XXXX}
by $\widehat{\mb{\Sigma}}_{(\widehat{A}_j^+)^c,\widehat{A}_j}$
gives
\[
\widehat{\mb{\Sigma}}_{(\widehat{A}_j^+)^c,\widehat{A}_j}
\widehat{\mb{\Omega}}_{\widehat{A}_j,j}
=-\widehat{\mb{\Omega}}_{jj}\widehat{\mb{\Sigma}}_{(\widehat{A}_j^+)^c,\widehat{A}_j} \widehat{\mb{\Sigma}}_{\widehat{A}_j,\widehat{A}_j} ^{-1}
\widehat{\mb{\Sigma}}_{\widehat{A}_j,j},
\]
where $\widehat{A}_j^+:=\widehat{A}_j\cup  \{j\}$ and $(\widehat{A}_j^+)^c:=[p]\setminus \widehat{A}_j^+$.
Then,
\begin{align}
\widehat{\mb{\Sigma}}_{(\widehat{A}_j^+)^c,\widehat{A}_j}
\widehat{\mb{\Omega}}_{\widehat{A}_j,j}
+\widehat{\mb{\Sigma}}_{(\widehat{A}_j^+)^c,j}\widehat{\mb{\Omega}}_{jj}
&=-\widehat{\mb{\Omega}}_{jj}\widehat{\mb{\Sigma}}_{(\widehat{A}_j^+)^c,\widehat{A}_j} \widehat{\mb{\Sigma}}_{\widehat{A}_j,\widehat{A}_j} ^{-1}
\widehat{\mb{\Sigma}}_{\widehat{A}_j,j}
+\widehat{\mb{\Sigma}}_{(\widehat{A}_j^+)^c,j}\widehat{\mb{\Omega}}_{jj}\nonumber\\
&=\widehat{\mb{\Omega}}_{jj}(\widehat{\mb{\Sigma}}_{(\widehat{A}_j^+)^c,j}
-\widehat{\mb{\Sigma}}_{(\widehat{A}_j^+)^c,\widehat{A}_j} \widehat{\mb{\Sigma}}_{\widehat{A}_j,\widehat{A}_j} ^{-1}
\widehat{\mb{\Sigma}}_{\widehat{A}_j,j})\nonumber\\
&=\widehat{\mb{\Omega}}_{jj}
n^{-1}\mb{X}_{*(\widehat{A}_j^+)^c}^\top [\mb{X}_{*j}-
\mb{X}_{*\widehat{A}_j}(\widehat{\bd{\alpha}}_j)_{\widehat{A}_j}].
\nonumber
\end{align}
Combining the above equation, \eqref{Sigma Theta=1} and \eqref{hat: ThetaAj=-Thetajj XXXX} obtains
\be\label{Sigma_hat Theta_hat = e +}
P(\widehat{\mb{\Sigma}}\widehat{\mb{\Omega}}_{*j}=\bd{e}_j+\widehat{\mb{\Omega}}_{jj}\widehat{\bd{d}}_j,~~\forall j\in[p])
=1-O(p^{-C}),
\ee
where $\bd{e}_j\in \mathbb{R}^p$ has $1$ on the $j$-th entry and 0
elsewhere, and
$\widehat{\bd{d}}_j=\mb{P}
(\bd{0}_{1\times \widehat{A}_j^+} , \{n^{-1}\mb{X}_{*(\widehat{A}_j^+)^c}^\top [\mb{X}_{*j}-
\mb{X}_{*\widehat{A}_j}(\widehat{\bd{\alpha}}_j)_{\widehat{A}_j}]\}^\top )^\top$
with a permutation matrix $\mb{P}$.

From \eqref{Sigma_hat Theta_hat = e +} and
$\mb{\Sigma}\mb{\Omega}_{*j}=\bd{e}_j$,
it holds with probability $1-O(p^{-C})$ that, for all $i,j\in[p]$,
\begin{align}
&\widehat{\mb{\Omega}}_{ij}-\widehat{\mb{\Omega}}_{*i}^\top
\widehat{\mb{\Omega}}_{jj}\widehat{\bd{d}}_j
-\mb{\Omega}_{ij}=\widehat{\mb{\Omega}}_{ij}-\widehat{\mb{\Omega}}_{*i}^\top(\widehat{\mb{\Sigma}}\widehat{\mb{\Omega}}_{*j}-\bd{e}_j)  -\mb{\Omega}_{ij}\nonumber\\
&\quad=-\mb{\Omega}_{*i}^\top(\widehat{\mb{\Sigma}}-\mb{\Sigma})
\mb{\Omega}_{*j}
-(\widehat{\mb{\Omega}}_{*j}-\mb{\Omega}_{*j})^\top
(\widehat{\mb{\Sigma}}\mb{\Omega}_{*i}-\bd{e}_i)
-(\widehat{\mb{\Omega}}_{*i}-\mb{\Omega}_{*i})^\top
(\widehat{\mb{\Sigma}}\widehat{\mb{\Omega}}_{*j}-\bd{e}_j).
\label{Theta-Theta2d-Theta}
\end{align}

First, consider the bound of the third term on the RHS of \eqref{Theta-Theta2d-Theta}. 
From \eqref{Sigma_hat Theta_hat = e +}, 
with probability $1-O(p^{-C})$ we have that
\begin{align}\label{3rd term in {Theta-Theta2d-Theta}}
|(\widehat{\mb{\Omega}}_{*i}-\mb{\Omega}_{*i})^\top
(\widehat{\mb{\Sigma}}\widehat{\mb{\Omega}}_{*j}-\bd{e}_j)|
&\le \| \widehat{\mb{\Omega}}_{*i}-\mb{\Omega}_{*i} \|_2
 \|(\widehat{\mb{\Sigma}}\widehat{\mb{\Omega}}_{*j}-\bd{e}_j)_{\widehat{A}_i^+\cup A_i^*}\|_2\nonumber\\
&=
\| \widehat{\mb{\Omega}}_{*i}-\mb{\Omega}_{*i} \|_2\widehat{\mb{\Omega}}_{jj}\|( \widehat{\bd{d}}_j)_{\widehat{A}_i^+\cup A_i^*}\|_2,~~~\forall i,j\in[p].
\end{align}
{\color{red}Due to the first part of Lemma~\ref{Lemma: A_hat=A}, i.e.,
\be\label{A_j in A_j_hat with high prob}
\inf_{\mb{\Omega}\in\mathcal{G}(T+1)}P(A_j^*\subseteq\widehat{A}_j, \forall j\in [p])=1-O(p^{-C}),
\ee
}
it holds with probability $1-O(p^{-C})$ that, for all $i,j\in[p]$,
\begin{align}\label{norm-2 of d_j_hat}
&\|( \widehat{\bd{d}}_j)_{\widehat{A}_i^+\cup A_i^*}\|_2=
\|n^{-1}(\mb{X}^\top)_{(\widehat{A}_j^+)^c\cap (\widehat{A}_i^+\cup A_i^*),*} [\mb{X}_{*j}-
\mb{X}_{*\widehat{A}_j}(\widehat{\bd{\alpha}}_j)_{\widehat{A}_j}]\|_2\nonumber\\
&=
\|n^{-1}(\mb{X}^\top)_{(\widehat{A}_j^+)^c\cap \widehat{A}_i^+,*} [\mb{X}_{*j}-
\mb{X}_{*\widehat{A}_j}(\widehat{\bd{\alpha}}_j)_{\widehat{A}_j}]\|_2
\nonumber\\
&\le 
\|
n^{-1}(\mb{X}^\top)_{(\widehat{A}_j^+)^c\cap\widehat{A}_i^+,*}  [\mb{X}_{*j}-
\mb{X}_{*\widehat{A}_j}(\bd{\alpha}_j^*)_{\widehat{A}_j}]
 \|_2\nonumber\\
 &\qquad+\|n^{-1} (\mb{X}^\top)_{(\widehat{A}_j^+)^c\cap\widehat{A}_i^+,*} \mb{X}_{*\widehat{A}_j}[(\bd{\alpha}_j^*)_{\widehat{A}_j}-(\widehat{\bd{\alpha}}_j)_{\widehat{A}_j}] \|_2\nonumber\\
& \le \| n^{-1}(\mb{X}^\top)_{(\widehat{A}_j^+)^c\cap\widehat{A}_i^+,*} \bd{\epsilon}_j\|_2
+\|n^{-1} (\mb{X}^\top)_{(\widehat{A}_j^+)^c\cap\widehat{A}_i^+,*} \mb{X}_{*\widehat{A}_j} \|_2
\|(\bd{\alpha}_j^*)_{\widehat{A}_j}-(\widehat{\bd{\alpha}}_j)_{\widehat{A}_j} \|_2.
\end{align}
Similar to \eqref{bound for xe},
\begin{align}\label{1st term in norm-2 of d_j_hat}
P\Big(\max_{1\le i,j\le p}\| n^{-1}(\mb{X}^\top)_{(\widehat{A}_j^+)^c\cap\widehat{A}_i^+,*} \bd{\epsilon}_j\|_2
&\le \max_{1\le j\le p}
\max_{A\subseteq [p]\setminus \{j\}, |A|\le T+1}
\| n^{-1}(\mb{X}^\top)_{A*} \bd{\epsilon}_j\|_2
\nonumber\\
&\le M_{22}\sqrt{T (\log p)/n}\Big)=1-O(p^{-C})
\end{align}
for some  constant $M_{22}>0$ dependent on $C$.
By  \eqref{spectrum of Sigma_AA} (which holds with $S_{2T}$ replaced by $S_{2T+1}$),
\begin{align}\label{2.1 term in norm-2 of d_j_hat}
P\Big(\max_{1\le i,j\le p}\|n^{-1} (\mb{X}^\top)_{(\widehat{A}_j^+)^c\cap\widehat{A}_i^+,*} \mb{X}_{*\widehat{A}_j} \|_2
&\le 
\max_{1\le i,j\le p}\|\widehat{\mb{\Sigma}}_{\widehat{A}_i^+\cup \widehat{A}_j, \widehat{A}_i^+\cup \widehat{A}_j} \|_2
\nonumber\\
&\le \max_{A\in S_{2T+1} }\|\widehat{\mb{\Sigma}}_{AA} \|_2
\le 
2\kappa
\Big)
=1-O(p^{-C}).
\end{align}
From \eqref{hat: ThetaAj=-Thetajj XXXX} and \eqref{A_j in A_j_hat with high prob}, it holds with probability $1-O(p^{-C})$ that, for all $j\in[p]$,
\begin{align}
&(\widehat{\bd{\alpha}}_j)_{\widehat{A}_j}-(\bd{\alpha}_j^*)_{\widehat{A}_j}
=-\widehat{\mb{\Omega}}_{jj}^{-1}\widehat{\mb{\Omega}}_{\widehat{A}_j, j}-
(\bd{\alpha}_j^*)_{\widehat{A}_j}
\nonumber\\
&=
(\mb{X}_{*\widehat{A}_j}^\top \mb{X}_{* \widehat{A}_j})^{-1} \mb{X}_{*\widehat{A}_j}^\top [\mb{X}_{*j}-  \mb{X}_{*\widehat{A}_j} (\bd{\alpha}_j^*)_{\widehat{A}_j} ]     
=(\mb{X}_{*\widehat{A}_j}^\top \mb{X}_{* \widehat{A}_j})^{-1} \mb{X}_{*\widehat{A}_j}^\top\bd{\epsilon}_j
%\quad \text{\color{red}(since $A_j^*\subseteq \widehat{A}_j$ due to \eqref{support contain})}
\nonumber\\
&=
\mb{\Sigma}_{\widehat{A}_j\widehat{A}_j}^{-1}\mb{X}_{*\widehat{A}_j}^\top\bd{\epsilon}_j/n
+
 \underbrace{(\widehat{\mb{\Sigma}}_{\widehat{A}_j\widehat{A}_j}^{-1} -\mb{\Sigma}_{\widehat{A}_j\widehat{A}_j}^{-1})\mb{X}_{*\widehat{A}_j}^\top\bd{\epsilon}_j/n}_{\displaystyle \bd{r}_{1j}}.
\label{diff in alpha}
\end{align}
Note that
\begin{align}
\| \bd{r}_{1j}\|_2&=\left\|(\widehat{\mb{\Sigma}}_{\widehat{A}_j\widehat{A}_j}^{-1} -\mb{\Sigma}_{\widehat{A}_j\widehat{A}_j}^{-1})\mb{X}_{*\widehat{A}_j}^\top\bd{\epsilon}_j/n\right\|_2
=  \left\|\mb{\Sigma}_{\widehat{A}_j\widehat{A}_j}^{-1}(\mb{\Sigma}_{\widehat{A}_j\widehat{A}_j}-\widehat{\mb{\Sigma}}_{\widehat{A}_j\widehat{A}_j})\widehat{\mb{\Sigma}}_{\widehat{A}_j\widehat{A}_j}^{-1}\mb{X}_{*\widehat{A}_j}^\top\bd{\epsilon}_j/n\right\|_2\nonumber\\
&\le 
 \left\|\mb{\Sigma}_{\widehat{A}_j\widehat{A}_j}^{-1}(\mb{\Sigma}_{\widehat{A}_j\widehat{A}_j}-\widehat{\mb{\Sigma}}_{\widehat{A}_j\widehat{A}_j})\right\|_2\left\| \widehat{\mb{\Sigma}}_{\widehat{A}_j\widehat{A}_j}^{-1}\mb{X}_{*\widehat{A}_j}^\top\bd{\epsilon}_j/n\right\|_2\nonumber\\
 &\le \max_{j\in[p]} \max_{A\subseteq [p]\setminus \{j\}: |A|\le T}
 \left\|\mb{\Sigma}_{AA}^{-1}(\mb{\Sigma}_{AA}-\widehat{\mb{\Sigma}}_{AA})\right\|_2\left\| \widehat{\mb{\Sigma}}_{AA}^{-1}\mb{X}_{*A}^\top\bd{\epsilon}_j/n\right\|_2.
 \label{2nd term in diff in alpha, ineq1}
\end{align}
From \eqref{bounded Sigma_AA}
and \eqref{Sigma_AA's error in norm 2},
\begin{align}
 \max_{j\in[p]} \max_{A\subseteq [p]\setminus \{j\}: |A|\le T}\left\|\mb{\Sigma}_{AA}^{-1}(\mb{\Sigma}_{AA}-\widehat{\mb{\Sigma}}_{AA})\right\|_2
 &\le (\max_{A\in S_T}\| \mb{\Sigma}_{AA}^{-1}\|_2)
 (\max_{A\in S_T}\| \mb{\Sigma}_{AA}-\widehat{\mb{\Sigma}}_{AA}\|_2)\nonumber\\
 & \le \kappa \check{M}_0\sqrt{T(\log p)/n}
  \label{Sigma^-1(diff in Sigma)}
\end{align}
with probability $1-O(p^{-C})$.
For the second component on the RHS of \eqref{2nd term in diff in alpha, ineq1}, 
by \eqref{spectrum of Sigma_AA} and 
an inequality similar to \eqref{bound for xe},
\begin{align}
 \max_{j\in[p]} \max_{A\subseteq [p]\setminus \{j\}: |A|\le T}  \left\| \widehat{\mb{\Sigma}}_{AA}^{-1}\mb{X}_{*A}^\top\bd{\epsilon}_j/n\right\|_2
&\le \max_{j\in[p]} \max_{A\subseteq [p]\setminus \{j\}: |A|\le T}\| \widehat{\mb{\Sigma}}_{AA}^{-1} \|_2 \| \mb{X}_{*A}^\top\bd{\epsilon}_j/n \|_2\nonumber\\
&\le \tilde{M}_4  \sqrt{T(\log p)/n}
\label{2nd term in diff in alpha, ineq2}
\end{align}
with probability $1-O(p^{-C})$ for some constant  $\tilde{M}_4>0$ dependent on $C$.
Plugging \eqref{Sigma^-1(diff in Sigma)} and \eqref{2nd term in diff in alpha, ineq2} into \eqref{2nd term in diff in alpha, ineq1} yields
\be\label{2nd term in diff in alpha}
\max_{j\in[p]} \|\bd{r}_{1j}\|_\infty\le \max_{j\in[p]} \|\bd{r}_{1j}\|_2\le \kappa\check{M}_0  \tilde{M}_4 T(\log p)/n
\ee
with probability $1-O(p^{-C})$.
Similar to \eqref{2nd term in diff in alpha, ineq2},
the first term on the RHS of \eqref{diff in alpha}
\[
\max_{j\in[p]}\|\mb{\Sigma}_{\widehat{A}_j\widehat{A}_j}^{-1}\mb{X}_{*\widehat{A}_j}^\top\bd{\epsilon}_j/n\|_2
\le \max_{j\in[p]} \max_{A\subseteq [p]\setminus \{j\}: |A|\le T}
\|\mb{\Sigma}_{AA}^{-1}\mb{X}_{*A}^\top\bd{\epsilon}_j/n\|_2
\le  \tilde{M}_4\sqrt{T (\log p)/n}
\]
with probability $1-O(p^{-C})$.
By  the above inequality, \eqref{2nd term in diff in alpha}, \eqref{diff in alpha}, and $T(\log p)/n=o(1)$,
\be\label{2.2 term in norm-2 of d_j_hat}
P\left(\max_{j\in[p]}\|(\widehat{\bd{\alpha}}_j)_{\widehat{A}_j}-(\bd{\alpha}_j^*)_{\widehat{A}_j}\|_2
\le M_\alpha\sqrt{T (\log p)/n}\right)=1-O(p^{-C})
\ee
for some constant $M_\alpha>0$ dependent on $C$.
Combining \eqref{norm-2 of d_j_hat}, \eqref{1st term in norm-2 of d_j_hat},
\eqref{2.1 term in norm-2 of d_j_hat} and \eqref{2.2 term in norm-2 of d_j_hat} yields
\be\label{bound for norm-2 of d_j_hat}
P\left(
\max_{1\le i,j\le p}\|( \widehat{\bd{d}}_j)_{\widehat{A}_i^+\cup A_i^*}\|_2
\le (M_{22}+2\kappa M_\alpha) \sqrt{T(\log p)/n}
\right)=1-O(p^{-C}).
\ee
Plugging \eqref{Theta column norm 2}, \eqref{bound for Theta_jj} 
 and \eqref{bound for norm-2 of d_j_hat}
into \eqref{3rd term in {Theta-Theta2d-Theta}} yields
\be\label{bound for the 3rd term in {Theta-Theta2d-Theta}}
P\left(\max_{1\le i,j\le p}|(\widehat{\mb{\Omega}}_{*i}-\mb{\Omega}_{*i})^\top
(\widehat{\mb{\Sigma}}\widehat{\mb{\Omega}}_{*j}-\bd{e}_j)|
\le M_{31} T(\log p)/n
\right)=1-O(p^{-C})
\ee
for some constant $M_{31}>0$ dependent on $C$.

For the second term on the RHS of \eqref{Theta-Theta2d-Theta},
by \eqref{norm-1 bound for Theta}, Lemma~10 in the supplement of \citet{Jank17}, $\max_{i\in[p]}\|\mb{\Omega}_{*i}\|_2\le\|\mb{\Omega}\|_2\le \kappa$, and $(\log p)/n=o(1)$,
we have
\be\label{2nd term in Theta-Theta2d-Theta}
\max_{1\le i,j\le p}|(\widehat{\mb{\Omega}}_{*j}-\mb{\Omega}_{*j})^\top
(\widehat{\mb{\Sigma}}\mb{\Omega}_{*i}-\bd{e}_i)|
\le \max_{j\in[p]}\| \widehat{\mb{\Omega}}_{*j}-\mb{\Omega}_{*j}\|_1  \max_{i\in[p]}\| \widehat{\mb{\Sigma}}\mb{\Omega}_{*i}-\bd{e}_i\|_{\max}
\le \check{M}_2 T (\log p)/n
\ee
with probability $1-O(p^{-C})$
for some constant $\check{M}_2>0$ dependent on $C$.

It follows from \eqref{Theta-Theta2d-Theta}, \eqref{2nd term in Theta-Theta2d-Theta} and \eqref{bound for the 3rd term in {Theta-Theta2d-Theta}} that
$\max_{1\le i,j\le p}|\Delta_{ij}|=O_P(T(\log p)/\sqrt{n})$ holds uniformly for $\mb{\Omega}\in\mathcal{G}(T+1)$.

Using \eqref{Theta-Theta2d-Theta}, \eqref{2nd term in Theta-Theta2d-Theta} and \eqref{bound for the 3rd term in {Theta-Theta2d-Theta}},
and following the proof of Theorem~1 in \citet{Jank17}
yields the asymptotic normality.
\end{proof}

\begin{proof}[Proof of Lemma~\ref{thm: normality}]
Note that
\be\label{Theta_hat2d_hat}
\widehat{\mb{\Omega}}_{*i}^\top
\widehat{\mb{\Omega}}_{jj}\widehat{\bd{d}}_j
=\mb{\Omega}_{*i}^{\top}\mb{\Omega}_{jj}
\widehat{\bd{d}}_j
+
\underbrace{(\widehat{\mb{\Omega}}_{*i}^\top-\mb{\Omega}_{*i}^{\top})
\widehat{\mb{\Omega}}_{jj}\widehat{\bd{d}}_j
+\mb{\Omega}_{*i}^{\top}
(\widehat{\mb{\Omega}}_{jj}-\mb{\Omega}_{jj})\widehat{\bd{d}}_j}_{\displaystyle r_{2ij}}.
\ee

We first assume that $A_j^*\subseteq \widehat{A}_j$ for all $j\in [p]$.

Using \eqref{Sigma_hat Theta_hat = e +} and \eqref{bound for the 3rd term in {Theta-Theta2d-Theta}}
for the first term in $r_{2ij}$,
and using $\|\mb{\Omega}_{*i} \|_2\le \|\mb{\Omega} \|_2\le \kappa$,
\eqref{diff in Theta[j,j]} and \eqref{bound for norm-2 of d_j_hat}
for the second term in $r_{2ij}$,
we have
\be\label{bound r_2ij}
P\Big(\max_{1\le i,j\le p}|r_{2ij}|
\le M_{23}T(\log p)/n\Big)
=1-O(p^{-C})
\ee
for some  constant $M_{23}>0$ dependent on $C$.

Consider the first term on the RHS of \eqref{Theta_hat2d_hat}.
From \eqref{alpha=Sigma^-1 Sigma}, $P(\mb{X}_{*,\widehat{A}_j}^\top [\mb{X}_{*j}-
\mb{X}_{*\widehat{A}_j}(\widehat{\bd{\alpha}}_j)_{\widehat{A}_j}]=\bd{0},\forall j\in[p])=1-O(p^{-C})$.
From this equation,  \eqref{alpha=Sigma^-1 Sigma} 
and \eqref{Theta_jj_hat},
with probability $1-O(p^{-C})$
we have that
\be\label{Theta2d}
\mb{\Omega}_{*i}^{\top}\mb{\Omega}_{jj}
\widehat{\bd{d}}_j
=\mb{\Omega}_{*i}^{\top}\mb{\Omega}_{jj}  n^{-1}\mb{X}^\top [\mb{X}_{*j}-
\mb{X}_{*\widehat{A}_j}(\widehat{\bd{\alpha}}_j)_{\widehat{A}_j}]
-\mb{\Omega}_{ji}^{\top}\mb{\Omega}_{jj} \widehat{\mb{\Omega}}_{jj}^{-1},~~~\forall i,j\in[p].
%&=\mb{\Omega}_{(\widehat{A}_j^+)^c\cap A_i^+,i}^{*\top}\mb{\Omega}_{jj}n^{-1}\mb{X}_{*,(\widehat{A}_j^+)^c\cap A_i^+}^\top (\mb{X}_{*j}-\mb{X}_{*\widehat{A}_j}(\widehat{\bd{\alpha}}_j)_{\widehat{A}_j})
\ee
The first term on the RHS of~\eqref{Theta2d}
\begin{align}
&\mb{\Omega}_{*i}^{\top}\mb{\Omega}_{jj}  n^{-1}\mb{X}^\top [\mb{X}_{*j}-
\mb{X}_{*\widehat{A}_j}(\widehat{\bd{\alpha}}_j)_{\widehat{A}_j}]\nonumber\\
&=\mb{\Omega}_{*i}^{\top}\mb{\Omega}_{jj}  n^{-1}\mb{X}^\top [\mb{X}_{*j}-
\mb{X}_{*\widehat{A}_j}(\bd{\alpha}_j^*)_{\widehat{A}_j}]
-\mb{\Omega}_{*i}^{\top}\mb{\Omega}_{jj}  n^{-1}\mb{X}^\top\mb{X}_{*\widehat{A}_j} [(\widehat{\bd{\alpha}}_j)_{\widehat{A}_j}-(\bd{\alpha}_j^*)_{\widehat{A}_j}]\nonumber\\
&=\mb{\Omega}_{*i}^{\top}\mb{\Omega}_{jj}  n^{-1}\mb{X}^\top \bd{\epsilon}_j
-\mb{\Omega}_{*i}^{\top}\mb{\Omega}_{jj}  \mb{\Sigma}_{*\widehat{A}_j} [(\widehat{\bd{\alpha}}_j)_{\widehat{A}_j}-(\bd{\alpha}_j^*)_{\widehat{A}_j}]\nonumber\\
&\qquad + 
\mb{\Omega}_{*i}^{\top}\mb{\Omega}_{jj}  
(\mb{\Sigma}_{*\widehat{A}_j} -n^{-1}\mb{X}^\top\mb{X}_{*\widehat{A}_j})
[(\widehat{\bd{\alpha}}_j)_{\widehat{A}_j}-(\bd{\alpha}_j^*)_{\widehat{A}_j}]\nonumber\\
&=\mb{\Omega}_{*i}^{\top}\mb{\Omega}_{jj}  n^{-1}\mb{X}^\top \bd{\epsilon}_j
-\mb{\Omega}_{jj} 
\mb{\Omega}_{*i}^{\top}  \mb{\Sigma}_{*\widehat{A}_j} 
\mb{\Sigma}_{\widehat{A}_j\widehat{A}_j}^{-1}\mb{X}_{*\widehat{A}_j}^\top\bd{\epsilon}_j/n\nonumber\\
&\qquad 
 \underbrace{ -\mb{\Omega}_{jj} (\bd{e}_i)_{\widehat{A}_j}^\top
\bd{r}_{1j}
 + 
\mb{\Omega}_{jj}  
[(\bd{e}_i)^\top_{\widehat{A}_j} -n^{-1}\mb{\Omega}_{ii}\bd{\epsilon}_i^\top\mb{X}_{*\widehat{A}_j}]
[(\widehat{\bd{\alpha}}_j)_{\widehat{A}_j}-(\bd{\alpha}_j^*)_{\widehat{A}_j}]}_{\displaystyle r_{3ij}}
\label{1st term in Theta2d_hat}
\end{align}
holds for all $i,j\in[p]$
with probability $1-O(p^{-C})$,
where the last equality follows from \eqref{diff in alpha} and \eqref{Theta_-jj formula}.
Using \eqref{2nd term in diff in alpha} for the first term in $r_{3ij}$,
and using 
a concentration inequality similar to \eqref{bound for xe}
together with \eqref{2.2 term in norm-2 of d_j_hat}
for the second term in $r_{3ij}$,
we have
\be\label{bound r_3ij}
P\Big(\max_{1\le i,j\le p}|r_{3ij}|
\le M_{34}T(\log p)/n\Big)
=1-O(p^{-C})
\ee
for some constant $M_{34}>0$ dependent on $C$.

Note that
\begin{align}
&-\mb{\Omega}_{*i}^{\top}(\widehat{\mb{\Sigma}}-\mb{\Sigma})
\mb{\Omega}_{*j}+
\mb{\Omega}_{*i}^{\top}\mb{\Omega}_{jj}  n^{-1}\mb{X}^\top \bd{\epsilon}_j
-\mb{\Omega}_{ji}^{\top}\mb{\Omega}_{jj} \widehat{\mb{\Omega}}_{jj}^{-1}
\nonumber\\
&=-\mb{\Omega}_{*i}^{\top}n^{-1}\mb{X}^\top\mb{X}\mb{\Omega}_{*j}
+
\mb{\Omega}_{*i}^{\top}\mb{\Omega}_{jj}  n^{-1}\mb{X}^\top \bd{\epsilon}_j
+\mb{\Omega}_{ji}^{\top}
-\mb{\Omega}_{ji}^{\top}\mb{\Omega}_{jj} \widehat{\mb{\Omega}}_{jj}^{-1}\nonumber\\
&=
-\mb{\Omega}_{jj}\mb{\Omega}_{*i}^{\top}n^{-1}\mb{X}^\top(\bd{\epsilon}_j-\bd{\epsilon}_j)
-\mb{\Omega}_{ij}\mb{\Omega}_{jj}(\widehat{\mb{\Omega}}_{jj}^{-1}-\mb{\Omega}_{jj}^{-1})\nonumber\\
&=-\mb{\Omega}_{ij}\mb{\Omega}_{jj}(\widehat{\mb{\Omega}}_{jj}^{-1}-\mb{\Omega}_{jj}^{-1})
\label{ee/n-sigma in AN}
\end{align}

By \eqref{Theta-Theta2d-Theta}, \eqref{2nd term in Theta-Theta2d-Theta}, \eqref{bound for the 3rd term in {Theta-Theta2d-Theta}},
\eqref{Theta_hat2d_hat}, \eqref{bound r_2ij}, \eqref{Theta2d}, 
\eqref{1st term in Theta2d_hat}, \eqref{bound r_3ij}, \eqref{ee/n-sigma in AN}, \eqref{Theta^-1 decomp}, \eqref{max z diff beta} and \eqref{3rd term in Theta^-1 bound},
%don't change the order of the equations!
we have 
\begin{align*}\label{supp: diff in Omega for AN}
\sqrt{n}(\widehat{\mb{\Omega}}_{ij}-\mb{\Omega}_{ij})
&=-\sqrt{n}\mb{\Omega}_{jj}\mb{\Omega}_{*i}^{\top}[\mb{\Sigma}_{*\widehat{A}_j}\mb{\Sigma}_{\widehat{A}_j\widehat{A}_j}^{-1}\mb{X}_{*\widehat{A}_j}^\top\bd{\epsilon}_j/n
]
-\sqrt{n}\mb{\Omega}_{jj}\mb{\Omega}_{ij}(\bd{\epsilon}_j^\top\bd{\epsilon}_j/n-\mb{\Omega}_{jj}^{-1})
+r_{ij}\nonumber\\
&=-\sqrt{n}[\mb{\Omega}_{jj}\mb{\Omega}_{ii}(\bd{\epsilon}_{i\parallel \bd{x}_{\widehat{A}_j}}^\top\bd{\epsilon}_j/n
+\bd{\epsilon}_{i\parallel\epsilon_j}^\top \bd{\epsilon}_j/n) - \mb{\Omega}_{ij}]+r_{ij},
\end{align*}
where
\be\label{bound r_ij}
P\Big(\max_{1\le i,j\le p}|r_{ij}|\le M_r T(\log p)/\sqrt{n}\Big)=1-O(p^{-C})
\ee
for some constant $M_r>0$ dependent on $C$,
$\bd{\epsilon}_{i\parallel \bd{x}_{\widehat{A}_j}}
:=[\mb{\Omega}_{ii}^{-1}(\bd{e}_i)_{\widehat{A}_j}^\top
\mb{\Sigma}_{\widehat{A}_j\widehat{A}_j}^{-1}\mb{X}_{*\widehat{A}_j}^\top]^\top$
consists of 
the $n$
observations of 
\begin{align*}
\epsilon_{i \parallel \bd{x}_{\widehat{A}_j}}
&:=E[\epsilon_i \bd{b}_{\widehat{A}_j}^\top|\widehat{A}_j]\bd{b}_{\widehat{A}_j}
=\mb{\Omega}_{ii}^{-1}E[\mb{\Omega}_{*i}^{\top}\bd{x} \bd{b}_{\widehat{A}_j}^\top|\widehat{A}_j]
\bd{b}_{\widehat{A}_j}
=\mb{\Omega}_{ii}^{-1}
\mb{\Omega}_{*i}^{\top}E[\bd{x}\bd{x}_{\widehat{A}_j}^\top|\widehat{A}_j]
\mb{\Sigma}_{\widehat{A}_j\widehat{A}_j}^{-1/2}\mb{\Sigma}_{\widehat{A}_j\widehat{A}_j}^{-1/2}\bd{x}_{\widehat{A}_j}\\
&=\mb{\Omega}_{ii}^{-1}\mb{\Omega}_{*i}^{\top}\mb{\Sigma}_{*\widehat{A}_j}\mb{\Sigma}_{\widehat{A}_j\widehat{A}_j}^{-1}\bd{x}_{\widehat{A}_j}
\qquad \text{with}\qquad \bd{b}_{\widehat{A}_j}:=\mb{\Sigma}_{\widehat{A}_j\widehat{A}_j}^{-1/2}\bd{x}_{\widehat{A}_j},
\end{align*}
and
$\bd{\epsilon}_{i\parallel\epsilon_j}:=\mb{\Omega}_{ii}^{-1}\mb{\Omega}_{ij}\bd{\epsilon}_j$ consists of the $n$ observations of
\begin{align*}
\epsilon_{i\parallel\epsilon_j}
&:=E[\epsilon_i\epsilon_j/\sd(\epsilon_j)]\epsilon_j/\sd(\epsilon_j)
=E[\epsilon_i\epsilon_j]\epsilon_j\mb{\Omega}_{jj}
=E[\mb{\Omega}_{ii}^{-1}\mb{\Omega}_{i*}\bd{x}
\bd{x}^\top \mb{\Omega}_{*j} \mb{\Omega}_{jj}^{-1}
] \epsilon_j\mb{\Omega}_{jj}\\
&=\mb{\Omega}_{ii}^{-1}\mb{\Omega}_{i*}\mb{\Sigma}
\mb{\Omega}_{*j} \mb{\Omega}_{jj}^{-1}\epsilon_j\mb{\Omega}_{jj}
=\mb{\Omega}_{ii}^{-1}\mb{\Omega}_{ij}\epsilon_j.
\end{align*}

We now connect \eqref{diff in Omega for AN}
and \eqref{diff in Omega for AN, simplified}
under the condition that 
$A_j^*\subseteq \widehat{A}_j$  for all $j\in [p]$.
We note that
$\bd{\epsilon}_{i\parallel \bd{x}_{\widehat{A}_j}}
+\bd{\epsilon}_{i\parallel\epsilon_j}
=\bd{\epsilon}_{i\parallel \bd{x}_{\widehat{A}_j^+}}$,
where 
$\bd{\epsilon}_{i\parallel \bd{x}_{\widehat{A}_j^+}}
=[\mb{\Omega}_{ii}^{-1}(\bd{e}_i)_{\widehat{A}_j^+}^\top
\mb{\Sigma}_{\widehat{A}_j^+\widehat{A}_j^+}^{-1}\mb{X}_{*\widehat{A}_j^+}^\top]^\top$.
To show this equation, first, 
like $\epsilon_{i\parallel \bd{x}_{\widehat{A}_j}}$,
we see that
$\epsilon_{i\parallel \bd{x}_{\widehat{A}_j^+}}
:=\mb{\Omega}_{ii}^{-1}(\bd{e}_i)_{\widehat{A}_j^+}^\top
\mb{\Sigma}_{\widehat{A}_j^+\widehat{A}_j^+}^{-1}\bd{x}_{\widehat{A}_j^+}$
is the orthogonal projection of $\epsilon_i$ onto $\lspan(\bd{x}_{\widehat{A}_j^+}^\top)$;
besides,
$\epsilon_{i \parallel \bd{x}_{\widehat{A}_j}}+\epsilon_{i\parallel\epsilon_j}$ is the orthogonal projection
of $\epsilon_i$ onto $\lspan((\bd{x}_{\widehat{A}_j}^\top, \epsilon_j))$ due to 
$\epsilon_j\perp \lspan(\bd{x}_{\widehat{A}_j}^\top)$;
then by $x_j=\bd{x}_{A_j^*}^\top (\bd{\alpha}_j^*)_{A_j^*}+\epsilon_j=\bd{x}_{\widehat{A}_j}^\top (\bd{\alpha}_j^*)_{\widehat{A}_j}+\epsilon_j$,
we have
$\lspan((\bd{x}_{\widehat{A}_j}^\top, \epsilon_j))=
\lspan(\bd{x}_{\widehat{A}_j^+}^\top)$,
and thus
$\epsilon_{i \parallel \bd{x}_{\widehat{A}_j}}+\epsilon_{i\parallel\epsilon_j}=\epsilon_{i\parallel \bd{x}_{\widehat{A}_j^+}}$.
On the other hand,
since the orthogonal projection of $x_j$
onto $\lspan(\bd{x}_{\setminus j}^\top)$
is
$\bd{x}_{\setminus j}^\top \bd{\alpha}_j^*=
\bd{x}_{A_j^*}^\top (\bd{\alpha}_j^*)_{A_j^*}
\in \lspan(\bd{x}_{A_j^*}^\top)\subseteq \lspan(\bd{x}_{\widehat{A}_j}^\top)$,
then
the orthogonal projection of $x_j$
onto $\lspan(\bd{x}_{\widehat{A}_j}^\top)$
is also $\bd{x}_{A_j^*}^\top (\bd{\alpha}_j^*)_{A_j^*}
=\bd{x}_{\widehat{A}_j}^\top (\bd{\alpha}_j^*)_{\widehat{A}_j}$.
Applying \eqref{lm, projection}-\eqref{Theta_-jj formula} to $\{x_\ell\}_{\ell\in \widehat{A}_j^+}$
yields 
$
(\mb{\Sigma}_{\widehat{A}_j^+\widehat{A}_j^+}^{-1})_{\widehat{A}_j,\hat{j}}=-(\mb{\Sigma}_{\widehat{A}_j^+\widehat{A}_j^+}^{-1})_{\hat{j}\hat{j}}(\bd{\alpha}_j^*)_{\widehat{A}_j}
$
and
$(\mb{\Sigma}_{\widehat{A}_j^+\widehat{A}_j^+}^{-1})_{\hat{j}\hat{j}}
=\mb{\Omega}_{jj}$.
It follows that
$\epsilon_j
=\mb{\Omega}_{jj}^{-1}(\bd{e}_j)_{\widehat{A}_j^+}^\top
\mb{\Sigma}_{\widehat{A}_j^+\widehat{A}_j^+}^{-1}\bd{x}_{\widehat{A}_j^+}$.
Then from \eqref{diff in Omega for AN}, we have
\[
 \sqrt{n}(\widehat{\mb{\Omega}}_{ij}-\mb{\Omega}_{ij})
=-\sqrt{n}
[(\mb{\Sigma}_{\widehat{A}_j^+\widehat{A}_j^+}^{-1})_{\hat{i}*}
\widehat{\mb{\Sigma}}_{\widehat{A}_j^+\widehat{A}_j^+}
(\mb{\Sigma}_{\widehat{A}_j^+\widehat{A}_j^+}^{-1})_{*\hat{j}}-\mb{\Omega}_{ij}]+r_{ij}~~\text{for}~~i\in \widehat{A}_j^+.
\]
Comparing the above equation with \eqref{diff in Omega for AN, simplified},
we obtain
\be\label{r=r tilde}
r_{ij}=\tilde{r}_{ij}~~\text{for all}~~i\in\widehat{A}_j^+, j\in [p].
\ee

Recall that 
 \eqref{r=r tilde} and \eqref{bound r_ij} are obtained by
assuming $A_j^*\subseteq \widehat{A}_j$ for all $j\in [p]$.
If $P(A_j^*\subseteq \widehat{A}_j,\forall j\in [p])\to 1$,
then 
$$P(r_{ij}=\tilde{r}_{ij}, \forall i\in\widehat{A}_j^+,j\in [p])\to 1$$
and 
\be\label{bound r_ij, eqn2}
P\Big(\max_{1\le i,j\le p}|r_{ij}|\le M_r T(\log p)/\sqrt{n}\Big)\to 1.
\ee

Note that 
\begin{align*}
\xi_{i,\widetilde{A}_j}
&:=-\mb{\Omega}_{ii}\mb{\Omega}_{jj}(\epsilon_{i\parallel \bd{x}_{\widetilde{A}_j}}+\epsilon_{i\parallel \epsilon_j})I(i\in \widetilde{A}_j)
-\mb{\Omega}_{jj}\mb{\Omega}_{ii}\epsilon_iI(i\notin \widetilde{A}_j)\\
&=[-\mb{\Omega}_{jj}(\bd{e}_i)_{\widetilde{A}_j}^\top \mb{\Sigma}_{\widetilde{A}_j\widetilde{A}_j}^{-1}\bd{x}_{\widetilde{A}_j}
-\mb{\Omega}_{jj}\mb{\Omega}_{ij}\epsilon_j]I(i\in \widetilde{A}_j)
-\mb{\Omega}_{jj}\mb{\Omega}_{ii}\epsilon_iI(i\notin \widetilde{A}_j),
\end{align*}
and
\begin{align*}
\bd{\xi}_{i,\widetilde{A}_j}^\top &:=-\mb{\Omega}_{ii}\mb{\Omega}_{jj}(\bd{\epsilon}_{i\parallel \bd{x}_{\widetilde{A}_j}}^\top+\bd{\epsilon}_{i\parallel \epsilon_j}^\top)I(i\in \widetilde{A}_j)
-\mb{\Omega}_{jj}\mb{\Omega}_{ii}\bd{\epsilon}_i^\top I(i\notin \widetilde{A}_j)\\
&= [-\mb{\Omega}_{jj}(\bd{e}_i)_{\widetilde{A}_j}^\top \mb{\Sigma}_{\widetilde{A}_j\widetilde{A}_j}^{-1}\mb{X}_{*\widetilde{A}_j}^\top
-\mb{\Omega}_{jj}\mb{\Omega}_{ij}\bd{\epsilon}_j^\top]I(i\in \widetilde{A}_j)
-\mb{\Omega}_{jj}\mb{\Omega}_{ii}\bd{\epsilon}_i^\top I(i\notin \widetilde{A}_j).
\end{align*}
Since $\widetilde{A}_j$ is a fixed set given $n$ and $p$, we have
$E[\xi_{i,\widetilde{A}_j}\epsilon_j]=-\mb{\Omega}_{ij}$
due to $\epsilon_j\perp \lspan(\bd{x}_{\setminus j}^\top) \ni\epsilon_{i\parallel \bd{x}_{\widetilde{A}_j}}$ and 
$E[\epsilon_{i\parallel\epsilon_j}\epsilon_j]=
E[\epsilon_i\epsilon_j]=\mb{\Omega}_{ii}^{-1}\mb{\Omega}_{jj}^{-1}\mb{\Omega}_{ij}$.
By the Lyapunov inequality, Minkowski inequality, 
Cauchy-Schwarz inequality,
and $\|\mb{\Omega}\|_{\max}\le \| \mb{\Omega}\|_2\le \kappa$, we have that, for $1\le i,j\le p$,
\begin{align}\label{eq: E|xie+Omega|3}
E[|\xi_{i,\widetilde{A}_j}\epsilon_j+\mb{\Omega}_{ij}|^3]
&\le \{ E[|\xi_{i,\widetilde{A}_j}\epsilon_j+\mb{\Omega}_{ij}|^4] \}^{\frac{3}{4}}
\le 
\big(
\{E[(\xi_{i,\widetilde{A}_j}\epsilon_j)^4]\}^{\frac{1}{4}}+
|\mb{\Omega}_{ij}|
\big)^3\nonumber\\
&\le 
\big(
\{E[\xi_{i,\widetilde{A}_j}^8] E[\epsilon_j^8]\}^{\frac{1}{8}}+
\kappa
\big)^3.
\end{align}

Define $\bd{y}_{\widetilde{A}_j}=\mb{\Sigma}_{\widetilde{A}_j\widetilde{A}_j}^{-1} \bd{x}_{\widetilde{A}_j}$, 
$\mb{Y}_{*\widetilde{A}_j}= \mb{X}_{*\widetilde{A}_j}\mb{\Sigma}_{\widetilde{A}_j\widetilde{A}_j}^{-1}$,
and $\bd{v}_{ij}= \mb{\Sigma}_{\widetilde{A}_j\widetilde{A}_j}^{-1}(\bd{e}_i)_{\widetilde{A}_j}$.
If $i\in \widetilde{A}_j$,
then let $i_{\widetilde{A}_j}$ denote the position of $i$
in the set $\widetilde{A}_j$
when its elements are sorted in ascending order.
From Cauchy’s interlace theorem and Assumption~\ref{Theta bound},
$\|\bd{v}_{ij} \|_2\le \| \mb{\Sigma}_{\widetilde{A}_j\widetilde{A}_j}^{-1} \|_2\| (\bd{e}_i)_{\widetilde{A}_j}\|_2\le \lambda^{-1}_{\min}(\mb{\Sigma}_{\widetilde{A}_j\widetilde{A}_j})\le\lambda^{-1}_{\min}(\mb{\Sigma})
=\lambda_{\max}(\mb{\Omega})\le \kappa$.
Then by Assumption~\ref{subGauss}, 
\begin{align}\label{subGauss y}
E\left[\exp\{(\bd{y}_{\widetilde{A}_j})_{\{i_{\widetilde{A}_j}\}}^2/(\kappa K)^2\}\right]
&=E\left[\exp(|\bd{v}_{ij}^\top\bd{x}_{\widetilde{A}_j}|^2/(\kappa K)^2)\right]
=E\left[\exp(|\kappa^{-1}\bd{v}_{ij}^\top\bd{x}_{\widetilde{A}_j}|^2/K^2)\right]\nonumber\\
&\le \sup_{\bd{v}\in \mathbb{R}^p:\|\bd{v}\|_2\le 1}E\left[\exp(|\bd{v}^\top\bd{x}|^2/K^2)\right]\le 2. 
\end{align}
Thus, $(\bd{y}_{\widetilde{A}_j})_{\{i_{\widetilde{A}_j}\}}$ is a sub-Gaussian random variable \citep[][Definition 2.5.6]{vershynin_2018} with a parameter $\kappa K$ for $i\in \widetilde{A}_j$ and $j\in [p]$.
From \eqref{subGauss epsilon},
$\epsilon_j$ is  a sub-Gaussian random variable with a parameter $\kappa K$ for $j\in[p]$.
By the bounds of sub-Gaussian moments \citep[][Proposition~2.5.2]{vershynin_2018}, 
we have
$$
\max\{E[\epsilon_j^8],E[(\bd{y}_{\widetilde{A}_j})_{\{i_{\widetilde{A}_j}\}}^8]\}\le C_{\epsilon y}^8
$$
with a constant $C_{\epsilon y}>0$ for $i\in \widetilde{A}_j$ and $j\in[p]$.
By
$\xi_{i,\widetilde{A}_j}=[-\mb{\Omega}_{jj}(\bd{y}_{\widetilde{A}_j})_{\{i_{\widetilde{A}_j}\}}
-\mb{\Omega}_{jj}\mb{\Omega}_{ij}\epsilon_j]I(i\in \widetilde{A}_j)
-\mb{\Omega}_{jj}\mb{\Omega}_{ii}\epsilon_iI(i\notin \widetilde{A}_j)$, 
the Minkowski inequality, 
and $\|\mb{\Omega}\|_{\max}\le \kappa$,
we have
\begin{align*}
E[\xi_{i,\widetilde{A}_j}^8]&\le \big(  \{E[(\mb{\Omega}_{jj}(\bd{y}_{\widetilde{A}_j})_{\{i_{\widetilde{A}_j}\}})^8]\}^{\frac{1}{8}} + \{E[(\mb{\Omega}_{jj}\mb{\Omega}_{ij}\epsilon_j)^8]\}^{\frac{1}{8}}  \big)^8
I(i\in \widetilde{A}_j)+
E[(\mb{\Omega}_{jj}\mb{\Omega}_{ii}\epsilon_i)^8]
I(i\notin \widetilde{A}_j)\\
&\le (  \kappa C_{\epsilon y} + \kappa^2  C_{\epsilon y} )^8
\end{align*}
for $1\le i,j\le p$.
Then by \eqref{eq: E|xie+Omega|3}, for $1\le i,j\le p$,
\[
E[|\xi_{i,\widetilde{A}_j}\epsilon_j+\mb{\Omega}_{ij}|^3]
\le 
\big(
\{E[\xi_{i,\widetilde{A}_j}^8] E[\epsilon_j^8]\}^{\frac{1}{8}}+
\kappa
\big)^3\le (  \kappa C_{\epsilon y}^2 + \kappa^2  C_{\epsilon y}^2+\kappa) ^3.
\]

By the above inequality,  $\min_{i\in \widehat{A}_j^+,j\in [p]}\sigma_{i,\widetilde{A}_j}\ge \omega$, and Berry-Esseen theorem (Theorem~3.4.17 in \citet{Durrett2019}), we have that,
for any $z\in\mathbb{R}$,
\begin{align}
&\max_{i\in\widehat{A}_j^+,j\in [p]}\left|P(\sqrt{n}(\bd{\xi}_{i,\widetilde{A}_j}^\top\bd{\epsilon}_j/n -E[\xi_{i,\widetilde{A}_j}\epsilon_j])/\sigma_{i,\widetilde{A}_j}\le z)-\Phi(z)\right|\nonumber\\
&\qquad \le 3 (  \kappa C_{\epsilon y}^2 + \kappa^2  C_{\epsilon y}^2+\kappa)^3/(\omega^3\sqrt{n})=:M_{13}/\sqrt{n}.
\nonumber
\end{align}
Let $Z$ be a standard Gaussian random variable, define the event $\mathcal{E}_{r}=\{\max_{1\le i,j\le p}|r_{ij}|\le M_r T(\log p)/\sqrt{n}\}$, and denote its complement by $\mathcal{E}_{r}^c$.
From $P(\cap_{j\in [p]}\mathcal{E}_j){\to 1}$,
the above inequality, \eqref{bound r_ij, eqn2}, and $T(\log p)/\sqrt{n}=o(1)$,
we have that, uniformly for all $z\in\mathbb{R}$,
\begin{align}\label{AN derive}
&\max_{i\in\widehat{A}_j^+,j\in[p]}\big|P(\sqrt{n}(\widehat{\mb{\Omega}}_{ij}-\mb{\Omega}_{ij})/\sigma_{i,\widehat{A}_j}\le z)
-\Phi(z)\big|\nonumber\\
&\le 
\max_{i\in\widehat{A}_j^+,j\in[p]}\big|P(\frac{\sqrt{n}(\widehat{\mb{\Omega}}_{ij}-\mb{\Omega}_{ij})-r_{ij}}{\sigma_{i,\widehat{A}_j}}\le z-\frac{r_{ij}}{\sigma_{i,\widehat{A}_j}})
-\Phi(z)\big|
\nonumber\\
&\le\max_{i\in\widehat{A}_j^+,j\in[p]}\Big(\big|
P(\{
\frac{\sqrt{n}(\widehat{\mb{\Omega}}_{ij}-\mb{\Omega}_{ij})-r_{ij}}{\sigma_{i,\widehat{A}_j}}\le z-\frac{r_{ij}}{\sigma_{i,\widehat{A}_j}}
\} \cap \mathcal{E}_j)-\Phi(z) \big|\nonumber\\
&\qquad\qquad\qquad+
P(\{\frac{\sqrt{n}(\widehat{\mb{\Omega}}_{ij}-\mb{\Omega}_{ij})-r_{ij}}{\sigma_{i,\widehat{A}_j}}\le z-\frac{r_{ij}}{\sigma_{i,\widehat{A}_j}}\} \cap \mathcal{E}_j^c)
\Big)
\nonumber\\
&\le \max_{i\in\widehat{A}_j^+,j\in[p]}\big|
P(\{
\frac{\sqrt{n}(\widehat{\mb{\Omega}}_{ij}-\mb{\Omega}_{ij})-r_{ij}}{\sigma_{i,\widehat{A}_j}}\le z-\frac{r_{ij}}{\sigma_{i,\widehat{A}_j}}
\} \cap \mathcal{E}_j)-\Phi(z) \big|
+o(1)
\nonumber\\
&\le
\max_{i\in\widehat{A}_j^+,j\in[p]}
\Big(\big|  P( \{\frac{\sqrt{n}(\bd{\xi}_{i,\widetilde{A}_j}^\top\bd{\epsilon}_j/n -E[\xi_{i,\widetilde{A}_j}\epsilon_j])}{\sigma_{i,\widetilde{A}_j}}\le z-\frac{r_{ij}}{\sigma_{i,\widetilde{A}_j}}\} \cap \mathcal{E}_j) -P(\{Z\le z-\frac{r_{ij}}{\sigma_{i,\widetilde{A}_j}}\})     \big|
\nonumber\\
&\qquad\qquad\qquad+
|P(\{Z\le z-\frac{r_{ij}}{\sigma_{i,\widetilde{A}_j}}\})- \Phi(z)| \Big)
+o(1)
\nonumber\\
&\le\max_{i\in\widehat{A}_j^+,j\in[p]}
\Big(\big|  P( \frac{\sqrt{n}(\bd{\xi}_{i,\widetilde{A}_j}^\top\bd{\epsilon}_j/n -E[\xi_{i,\widetilde{A}_j}\epsilon_j])}{\sigma_{i,\widetilde{A}_j}}\le z-\frac{r_{ij}}{\sigma_{i,\widetilde{A}_j}} ) -P(Z\le z-\frac{r_{ij}}{\sigma_{i,\widetilde{A}_j}})     \big|
\nonumber\\
&\qquad\qquad\qquad+
P( \{\frac{\sqrt{n}(\bd{\xi}_{i,\widetilde{A}_j}^\top\bd{\epsilon}_j/n -E[\xi_{i,\widetilde{A}_j}\epsilon_j])}{\sigma_{i,\widetilde{A}_j}}\le z-\frac{r_{ij}}{\sigma_{i,\widetilde{A}_j}}\} \cap \mathcal{E}_j^c)
\nonumber\\
&\qquad\qquad\qquad+
P(z-\frac{|r_{ij}|}{\sigma_{i,\widetilde{A}_j}}\le Z\le z+\frac{|r_{ij}|}{\sigma_{i,\widetilde{A}_j}}) \Big)
+o(1)
\nonumber\\
&\le M_{13}/\sqrt{n}+o(1)+\max_{i\in\widehat{A}_j^+,j\in[p]} \Big(P(\{z-\frac{|r_{ij}|}{\sigma_{i,\widetilde{A}_j}}\le Z\le z+\frac{|r_{ij}|}{\sigma_{i,\widetilde{A}_j}}\}\cap\mathcal{E}_r) 
\nonumber\\
&\qquad\qquad\qquad\qquad\qquad\qquad\qquad+
P(\{z-\frac{|r_{ij}|}{\sigma_{i,\widetilde{A}_j}}\le Z\le z+\frac{|r_{ij}|}{\sigma_{i,\widetilde{A}_j}}\}\cap\mathcal{E}_r^c) 
\Big)
\nonumber\\
&=M_{13}/\sqrt{n}+o(1)+
P(z-\frac{M_rT \log p}{\omega \sqrt{n}}\le Z\le z+\frac{M_rT \log p}{\omega \sqrt{n}})
\nonumber\\
&=o(1).
\end{align}

The proof is complete.
\end{proof}

\begin{proof}[Proof of Theorem~\ref{Corollary: AN}]
From Lemma~\ref{Lemma2: A_hat=A}, we have
$\inf_{\mb{\Omega}\in\mathcal{G}(s)} P(T_j^*=s_j-1, \forall j\in [p])\to 1$
and $\inf_{\mb{\Omega}\in\mathcal{G}(s)}P(\widehat{A}_j=A_j^*,\forall j\in [p])\to 1$.
Define event $\mathcal{E}_T=\{T_j^*=s_j-1, \forall j\in [p]\}$.
Then following the proof of Lemma~\ref{thm: normality}
and plugging
\begin{align*}
&P(z-\frac{M_rT \log p}{\omega \sqrt{n}}\le Z\le z+\frac{M_rT \log p}{\omega \sqrt{n}})\nonumber\\
&= 
P(\{z-\frac{M_rT \log p}{\omega \sqrt{n}}\le Z\le z+\frac{M_rT \log p}{\omega \sqrt{n}}\}\cap\mathcal{E}_T)\\
&\qquad+P(\{z-\frac{M_rT \log p}{\omega \sqrt{n}}\le Z\le z+\frac{M_rT \log p}{\omega \sqrt{n}}\}\cap\mathcal{E}_T^c)\\
&\le P(z-\frac{M_rs \log p}{\omega \sqrt{n}}\le Z\le z+\frac{M_rs \log p}{\omega \sqrt{n}})+P(\mathcal{E}_T^c)\\
&=o(1)~~~(\text{due to $s (\log p)/\sqrt{n}=o(1)$ and $P(\mathcal{E}_T^c)=o(1)$})
\end{align*}
 into the second to last line of \eqref{AN derive},
 we complete the proof. 
\end{proof}

\begin{proof}[Proof of Theorem~\ref{thm: consistency of sigma_iAj}]

First, we prove Part~\ref{thm (i): consistency of sigma_iAj} of the theorem.

When $A_j^*\subseteq \widetilde{A}_j$,
we have
\begin{align*}
\sigma_{i,\widetilde{A}_j}^2&:=\var(\xi_{i,\widetilde{A}_j}\epsilon_j)\\
&=\var\left[(\mb{\Sigma}_{\widetilde{A}_j^+\widetilde{A}_j^+}^{-1})_{\tilde{i}*}\bd{x}_{\widetilde{A}_j^+}\bd{x}_{\widetilde{A}_j^+}^\top(\mb{\Sigma}_{\widetilde{A}_j^+\widetilde{A}_j^+}^{-1})_{*\tilde{j}}I(i\in \widetilde{A}_j^+)
+
\mb{\Omega}_{*i}^\top\bd{x}_{\widetilde{A}_j^+}\bd{x}_{\widetilde{A}_j^+}^\top\mb{\Omega}_{*j}I(i\notin \widetilde{A}_j^+)
\right],
\end{align*}
where $\widetilde{A}_j^+:=\widetilde{A}_j\cup\{j\}$, and $\tilde{i}$ and $\tilde{j}$ are the positions of $i$ and $j$
in $\widetilde{A}_j^+$ 
when its elements are sorted in ascending order.
Define
\begin{align*}
\widehat{\sigma}_{i,\widetilde{A}_j}^2=&\frac{1}{n}\sum_{k=1}^n\Big[(\widehat{\mb{\Sigma}}_{\widetilde{A}_j^+\widetilde{A}_j^+}^{-1})_{\tilde{i}*}\mb{X}_{k,\widetilde{A}_j^+}^\top\mb{X}_{k,\widetilde{A}_j^+}(\widehat{\mb{\Sigma}}_{\widetilde{A}_j^+\widetilde{A}_j^+}^{-1})_{*\tilde{j}}I(i\in \widetilde{A}_j^+)
+
\widehat{\mb{\Omega}}_{*i}^\top\mb{X}_{k,\widetilde{A}_j^+}^\top\mb{X}_{k,\widetilde{A}_j^+}\widehat{\mb{\Omega}}_{*j}I(i\notin \widetilde{A}_j^+)
\Big]^2\\
&\quad -\frac{1}{2}(\widehat{\mb{\Omega}}_{ij}^2+\widehat{\mb{\Omega}}_{ji}^2) .
\end{align*}

From \eqref{Sigma_AA's error in norm 2},
\eqref{bounded Sigma_AA} and \eqref{spectrum of Sigma_AA},
with probability $1-O(p^C)$ we have
\begin{align}
\max_{j\in[p]}\| \widehat{\mb{\Sigma}}_{\widetilde{A}_j^+\widetilde{A}_j^+}^{-1}-\mb{\Sigma}_{\widetilde{A}_j^+\widetilde{A}_j^+}^{-1}  \|_2
&\le\max_{j\in[p]}\|\widehat{\mb{\Sigma}}_{\widetilde{A}_j^+\widetilde{A}_j^+}^{-1}\|_2\| \widehat{\mb{\Sigma}}_{\widetilde{A}_j^+\widetilde{A}_j^+}-\mb{\Sigma}_{\widetilde{A}_j^+\widetilde{A}_j^+}  \|_2\|\mb{\Sigma}_{\widetilde{A}_j^+\widetilde{A}_j^+}^{-1}\|_2\nonumber\\
&\le 2\kappa^2\check{M}_0\sqrt{T(\log p)/n}
\label{inv Simga_AA in norm-2}
\end{align}
and
\be\label{inv Simga_AA in norm-1}
\max_{j\in[p]}\| \widehat{\mb{\Sigma}}_{\widetilde{A}_j^+\widetilde{A}_j^+}^{-1}-\mb{\Sigma}_{\widetilde{A}_j^+\widetilde{A}_j^+}^{-1}  \|_1
\le \max_{j\in[p]}\| \widehat{\mb{\Sigma}}_{\widetilde{A}_j^+\widetilde{A}_j^+}^{-1}-\mb{\Sigma}_{\widetilde{A}_j^+\widetilde{A}_j^+}^{-1}  \|_2 \sqrt{T+1}
\le 2^{3/2}\kappa^2\check{M}_0T\sqrt{(\log p)/n}.
\ee

Following the proof of Lemma~3 in Supplementary Material of \citet{Jank17},
but using \eqref{inv Simga_AA in norm-2},
\eqref{inv Simga_AA in norm-1}, \eqref{norm-1 bound for Theta},  \eqref{Theta column norm 2}, \eqref{bounded Sigma_AA}, and 
\eqref{max norm err of Sigma} instead of their counterparts,
we obtain that, for all $\varepsilon>0$,
\[
P(\max_{i,j\in[p]} |\widehat{\sigma}_{i,\widetilde{A}_j}^2-\sigma_{i,\widetilde{A}_j}^2|\ge \varepsilon)=o(1)
~~\text{if}~~A_j^*\subseteq \widetilde{A}_j.
\]
Define $\mathcal{E}_A=\cap_{j\in [p]}\mathcal{E}_j$
and denote its complement by $\mathcal{E}_A^c$.
Then, we have
\begin{align}
&P(\max_{i\in \widehat{A}_j^+,j\in[p]} |\widehat{\sigma}_{i,\widehat{A}_j}^2-\sigma_{i,\widehat{A}_j}^2|\ge \varepsilon)\nonumber\\
&\le P\Big(\big\{\max_{i\in \widehat{A}_j^+,j\in[p]} |\widehat{\sigma}_{i,\widehat{A}_j}^2-\sigma_{i,\widehat{A}_j}^2|\ge \varepsilon\big\} \cap  \mathcal{E}_A\Big)
+P\Big(\big\{\max_{i\in \widehat{A}_j^+,j\in[p]} |\widehat{\sigma}_{i,\widehat{A}_j}^2-\sigma_{i,\widehat{A}_j}^2|\ge \varepsilon\big\} \cap  \mathcal{E}_A^c\Big)\nonumber\\
&= P\Big(\big\{\max_{i\in \widehat{A}_j^+,j\in[p]} |\widehat{\sigma}_{i,\widetilde{A}_j}^2-\sigma_{i,\widetilde{A}_j}^2|\ge \varepsilon\big\} \cap  \mathcal{E}_A\Big)  +o(1)\nonumber\\
&\le
P(\max_{i\in \widehat{A}_j^+,j\in[p]} |\widehat{\sigma}_{i,\widetilde{A}_j}^2-\sigma_{i,\widetilde{A}_j}^2|\ge \varepsilon)
+o(1)\nonumber\\
&=o(1).
\label{diff in sigma>e}
\end{align}

Now, we prove Part~\ref{thm (ii): consistency of sigma_iAj} of the theorem. 
Re-define 
\[
\widehat{\sigma}_{i,\widetilde{A}_j}^2
=[(\widehat{\mb{\Sigma}}_{\widetilde{A}_j^+\widetilde{A}_j^+}^{-1})_{\tilde{i}\tilde{i}}(\widehat{\mb{\Sigma}}_{\widetilde{A}_j^+\widetilde{A}_j^+}^{-1})_{\tilde{j}\tilde{j}}+(\widehat{\mb{\Sigma}}_{\widetilde{A}_j^+\widetilde{A}_j^+}^{-1})_{\tilde{i}\tilde{j}}^2]I(i\in \widetilde{A}_j^+)
+
(\widehat{\mb{\Omega}}_{ii}\widehat{\mb{\Omega}}_{jj}+\widehat{\mb{\Omega}}_{ij}^2)I(i\notin \widetilde{A}_j^+).
\]
Following the proof of Lemma~2 in Supplementary Material of \citet{Jank17},
but using \eqref{inv Simga_AA in norm-2}, \eqref{Theta column norm 2}, and \eqref{bounded Sigma_AA} instead of their counterparts,
we obtain that
\[
P(\max_{i\in \widehat{A}_j^+,j\in[p]} |\widehat{\sigma}_{i,\widetilde{A}_j}^2-\sigma_{i,\widetilde{A}_j}^2|\ge M \sqrt{T(\log p)/n})=O(p^{-C})
~~\text{if}~~A_j^*\subseteq \widetilde{A}_j.
\]
for some constant $M>0$ dependent on $C$.
Then using the same proof technique for \eqref{diff in sigma>e}
yields
\[
\max_{i\in \widehat{A}_j^+,j\in[p]} |\widehat{\sigma}_{i,\widehat{A}_j}^2-\sigma_{i,\widehat{A}_j}^2|
=O_P(\sqrt{T(\log p)/n}).
\]

The proof is complete.
\end{proof}

\begin{proof}[Proof of Theorem~\ref{thm: consistency of sigma_ij}]

Following the proof of Lemma~3 in Supplementary Material of \citet{Jank17},
but using  \eqref{norm-1 bound for Theta},  \eqref{Theta column norm 2}, \ref{Theta bound} and 
\eqref{max norm err of Sigma} instead of their counterparts,
we can obtain the result in Part~\ref{thm (i): consistency of sigma_ij} of the theorem.
On the other hand,
following the proof of Lemma~2 in Supplementary Material of \citet{Jank17},
but using \eqref{Theta column norm 2} and \ref{Theta bound} instead of their counterparts,
we can obtain the result in Part~\ref{thm (ii): consistency of sigma_ij}  of the theorem.
\end{proof}

\begin{proof}[Proof of Theorem~\ref{sigma_ij compare}]
For $i\in\widehat{A}_j^+$, we have
\[
\sigma_{i,\widehat{A}_j}^2=\var(\mb{\Omega}_{ii}\mb{\Omega}_{jj}  (\epsilon_{i\parallel \bd{x}_{\widehat{A}_j}}+\epsilon_{i\parallel \epsilon_j})\epsilon_j    |\widehat{A}_j)
= \var(\mb{\Omega}_{ii}\mb{\Omega}_{jj}  \epsilon_{i\parallel \bd{x}_{\widehat{A}_j^+}}\epsilon_j    |\widehat{A}_j)
=\mb{\Omega}_{ii}^2\mb{\Omega}_{jj}^2 E[\epsilon_{i\parallel \bd{x}_{\widehat{A}_j^+}}^2\epsilon_j^2  |\widehat{A}_j ]
-\mb{\Omega}_{ij}^2
\]
and
$
\sigma_{ij}^2=\var(\mb{\Omega}_{ii}\mb{\Omega}_{jj}  \epsilon_i\epsilon_j)
=\mb{\Omega}_{ii}^2\mb{\Omega}_{jj}^2 E[\epsilon_i^2\epsilon_j^2 ]
-\mb{\Omega}_{ij}^2.
$

If $i=j$,
then $\epsilon_{i\parallel \bd{x}_{\widehat{A}_j^+}}=\epsilon_i$,
and thus $\sigma_{i,\widehat{A}_j}^2=\sigma_{ij}^2$.

If $i\in \widehat{A}_j$,
we only need to compare 
$E[\epsilon_{i\parallel \bd{x}_{\widehat{A}_j^+}}^2\epsilon_j^2  |\widehat{A}_j ]$ 
and
$E[\epsilon_i^2\epsilon_j^2 ]$.
Define $\epsilon_{i\perp\bd{x}_{\widehat{A}_j^+}}=
\epsilon_i-\epsilon_{i\parallel \bd{x}_{\widehat{A}_j^+}}$,
which is the orthogonal rejection of $\epsilon_i$
from $\lspan( \bd{x}_{\widehat{A}_j^+}^\top)$
when $\widehat{A}_j$ is given.
Thus, $\epsilon_{i\perp\bd{x}_{\widehat{A}_j^+}}\perp\epsilon_{i\parallel\bd{x}_{\widehat{A}_j^+}}$.
Since $\epsilon_j\in \lspan( \bd{x}_{\widehat{A}_j^+}^\top)$
due to $A_j^*\subseteq \widehat{A}_j$,
we have $\epsilon_{i\perp\bd{x}_{\widehat{A}_j^+}}\perp \epsilon_j$.
Thus,
given $\widehat{A}_j$,
$\epsilon_{i\perp\bd{x}_{\widehat{A}_j^+}}$
is independent of $(\epsilon_{i\parallel\bd{x}_{\widehat{A}_j^+}},\epsilon_j)$ when $\bd{x}$ follows a $p$-variate Gaussian distribution.
Then, when $\widehat{A}_j$ is given, we have
\begin{align*}
E[\epsilon_i^2\epsilon_j^2 ]
&=E[\epsilon_i^2\epsilon_j^2 |\widehat{A}_j ]
=E[(\epsilon_{i\parallel \bd{x}_{\widehat{A}_j^+}}+\epsilon_{i\perp\bd{x}_{\widehat{A}_j^+}})^2\epsilon_j^2 |\widehat{A}_j]\\
&=E[\epsilon_{i\parallel \bd{x}_{\widehat{A}_j^+}}^2\epsilon_j^2|\widehat{A}_j ]
+E[\epsilon_{i\perp \bd{x}_{\widehat{A}_j^+}}^2\epsilon_j^2|\widehat{A}_j ]
+2E[\epsilon_{i\parallel \bd{x}_{\widehat{A}_j^+}}\epsilon_{i\perp \bd{x}_{\widehat{A}_j^+}}\epsilon_j^2|\widehat{A}_j ]\\
&=E[\epsilon_{i\parallel \bd{x}_{\widehat{A}_j^+}}^2\epsilon_j^2|\widehat{A}_j ]
+E[\epsilon_{i\perp \bd{x}_{\widehat{A}_j^+}}^2\epsilon_j^2|\widehat{A}_j ]
+2E[\epsilon_{i\parallel \bd{x}_{\widehat{A}_j^+}}\epsilon_j^2 |\widehat{A}_j ]E[ \epsilon_{i\perp \bd{x}_{\widehat{A}_j^+}}|\widehat{A}_j ]\\
&=E[\epsilon_{i\parallel \bd{x}_{\widehat{A}_j^+}}^2\epsilon_j^2|\widehat{A}_j ]
+E[\epsilon_{i\perp \bd{x}_{\widehat{A}_j^+}}^2\epsilon_j^2|\widehat{A}_j ]\\
&\ge E[\epsilon_{i\parallel \bd{x}_{\widehat{A}_j^+}}^2\epsilon_j^2|\widehat{A}_j ],
\end{align*}
where equality holds if and only if
$P(\epsilon_{i\perp \bd{x}_{\widehat{A}_j^+}} 
=\epsilon_i-\epsilon_{i\parallel \bd{x}_{\widehat{A}_j^+}}=0 |\widehat{A}_j )=1$, i.e.,
$(A_i^*\cup\{i\})\setminus \widehat{A}_j^+=\emptyset$.

The proof is complete.
\end{proof}

\begin{proof}[A counterexample to Theorem~\ref{sigma_ij compare} when $\bd{x}$ is non-Gaussian]
Given $\widehat{A}_j$, 
denote $X=\epsilon_{i\parallel \bd{x}_{\widehat{A}_j^+}}$, $Y=\epsilon_{i\perp\bd{x}_{\widehat{A}_j^+}}$,
and $Z=\epsilon_j$. 
Consider $i\in A_j^*$.
We have $E[(X+Y)Z]=E[\epsilon_i\epsilon_j]=\mb{\Omega}_{ii}^{-1}\mb{\Omega}_{jj}^{-1}\mb{\Omega}_{ij}\ne 0$.
Then by  $E[YZ]=E[\epsilon_{i\perp\bd{x}_{\widehat{A}_j^+}} \epsilon_j|\widehat{A}_j]=0$ (from the proof of Theorem~\ref{sigma_ij compare}),
we obtain $E[XZ]\ne 0$.
On the other hand, $E[(X+Y)^2]=\var(\epsilon_i)\ge E[X^2]=\var(\epsilon_{i\parallel \bd{x}_{\widehat{A}_j^+}}|\widehat{A}_j)$.
From the proof of Theorem~\ref{sigma_ij compare},
$\sigma_{ij}<\sigma_{i,\widehat{A}_j}$ is equivalent to 
$E[\epsilon_i^2\epsilon_j^2 ]<E[\epsilon_{i\parallel \bd{x}_{\widehat{A}_j^+}}^2\epsilon_j^2  |\widehat{A}_j ]$ when $i\in \widehat{A}_j\supseteq A_j^*$.
Thus, we only need to find a counterexample such that 
$E[(X+Y)^2Z^2]<E[X^2Z^2]$, $E[(X+Y)^2]\ge E[X^2]$, $E[XZ]\ne 0$, $E[XY]= 0$, $E[YZ]=0$, and $E[X]=E[Y]=E[Z]=0$.
Table~\ref{tab:counterexample} gives the joint probability distribution of $(X,Y,Z)$ for the counterexample.
\begin{table}[h!]
\centering
\begin{tabular}{ccccc}
\hline
$X$ & $Y$ & $Z$ & $P(X,Y,Z)$  \\
\hline
$1$ & $1$ & $1$ & $1/96$ \\
$-1$ & $1$ & $1$ &  $1/8$ \\
$-1$ & $-1$ & $1$ &  $1/8$ \\
$-1$ & $1$ & $-1$ &  $1/8$ \\
$1$ & $1$ & $-1$ &  $1/96$ \\
$1$ & $-1$ & $-1$ &  $3/16$ \\
$1$ & $-1$ & $1$ &  $1/16$ \\
$1$ & $1$ & $0$ & $11/48$\\
$-1$ & $-1$ & $0$ &  $1/8$ \\
\hline
\end{tabular}
\caption{The joint probability of $(X,Y,Z)$ for the counterexample.}
\label{tab:counterexample}
\end{table}
\end{proof}

\section{Additional simulation results}\label{sec: add siml}

 \begin{table}[!htb]
 \renewcommand\arraystretch{1.3}
 \begin{center}
\resizebox{\textwidth}{!}{
\begin{tabular}{cclcccc}
\toprule
Gaussian &   &  & $L_1$ norm& Spectral norm &Frobenius norm& Max norm  \\
Graph&   $p$& Method & $n=200/n=400$& $n=200/n=400$ &$n=200/n=400$& $n=200/n=400$  \\
\midrule

\multirow{26.5}{*}{Random}&\multirow{13}{*}{200} &
$L_0{:}~ \widehat{\mb{\Omega}}^{\text{S}}$& 4.122(0.920)/2.065(0.491)  & 2.204(0.518)/{\bf 1.158(0.282)} & 7.468(0.341)/3.563(0.244) & 1.423(0.526)/0.810(0.248)  \\
\multirow{13}{*}{}& \multirow{13}{*}{}& $L_0{:}~ \mathcal{T}(\widehat{\mb{\Omega}}^{\text{S}}|Z_0(\widehat{\mb{\Omega}}^\text{US}),S_L(\widehat{\mb{\Omega}}^\text{S}))$& 4.126(0.912)/{\bf 2.056(0.492)} & 2.228(0.502)/1.159(0.278) & 7.489(0.341)/{\bf 3.560(0.251)} & 1.434(0.518)/0.812(0.249)  \\
\multirow{13}{*}{}& \multirow{13}{*}{}& $L_0{:}~  \mathcal{T}(\widehat{\mb{\Omega}}^{\text{S}}|Z_0(\widehat{\mb{T}}),S_L(\widehat{\mb{\Omega}}^{\text{S}}))$& 4.226(0.956)/2.111(0.503) & 2.258(0.528)/1.184(0.285) & 7.520(0.350)/3.616(0.249) & 1.433(0.519)/0.834(0.253) \\
\multirow{13}{*}{}& \multirow{13}{*}{}& $L_0{:}~  \mathcal{T}(\widehat{\mb{T}}|Z_0(\widehat{\mb{T}}),S_L(\widehat{\mb{\Omega}}^{\text{S}}))$& 5.835(1.493)/2.843(0.751) & 2.835(0.683)/1.455(0.370) & 8.636(0.429)/4.189(0.302) & 1.458(0.507)/0.859(0.249) \\
\multirow{13}{*}{}& \multirow{13}{*}{}& $L_0{:}~  \mathcal{T}(\widehat{\mb{T}}|Z_0(\widehat{\mb{T}}),S_L(\widehat{\mb{T}}))$& 7.434(1.279)/3.330(0.884) & 3.351(0.536)/1.691(0.377) & 10.156(0.360)/5.045(0.369) & 1.532(0.447)/1.009(0.234) \\
\multirow{13}{*}{}& \multirow{13}{*}{}& $L_0{:}~  \widehat{\mb{T}}$& 40.199(3.800)/26.287(1.908) & 5.358(0.552)/3.483(0.216) & 24.011(0.389)/16.257(0.192) & 1.449(0.513)/0.816(0.248) \\
\bottomrule
\end{tabular}}
\caption{Average (standard deviation) of each matrix norm loss over 100 simulation replications under Gaussian random, hub, and cluster graph settings.}
\label{tab: norm loss, other Gaussian}
\end{center}
 \end{table}

   \begin{table}[!htb]
 \renewcommand\arraystretch{1.3}
 \begin{center}
\resizebox{\textwidth}{!}{
\begin{tabular}{cclcccc}
\toprule
Gaussian &   &  & $L_1$ norm& Spectral norm &Frobenius norm& Max norm  \\
Graph&   $p$& Method & $n=200/n=400$& $n=200/n=400$ &$n=200/n=400$& $n=200/n=400$  \\
\midrule

\multirow{26.5}{*}{Cluster}&\multirow{13}{*}{200} &
$L_0{:}~ \widehat{\mb{\Omega}}^{\text{S}}$& 6.040(1.529)/2.424(0.565)  & 3.422(0.748)/1.510(0.366) & 11.470(0.601)/4.438(0.273) & 1.737(0.514)/0.930(0.301) \\
\multirow{13}{*}{}& \multirow{13}{*}{}& $L_0{:}~ \mathcal{T}(\widehat{\mb{\Omega}}^{\text{S}}|Z_0(\widehat{\mb{\Omega}}^\text{US}),S_L(\widehat{\mb{\Omega}}^\text{S}))$& 6.007(1.512)/2.385(0.547) & {\bf 3.398(0.760)/1.492(0.366)} & 11.463(0.602)/4.417(0.273) & 1.740(0.512)/{\bf 0.929(0.303)}  \\
\multirow{13}{*}{}& \multirow{13}{*}{}& $L_0{:}~  \mathcal{T}(\widehat{\mb{\Omega}}^{\text{S}}|Z_0(\widehat{\mb{T}}),S_L(\widehat{\mb{\Omega}}^{\text{S}}))$&{\bf 5.987(1.523)/2.383(0.546)} & 3.406(0.742)/1.493(0.366) &{\bf 11.460(0.597)/4.417(0.272)} & 1.740(0.514)/{\bf 0.929(0.303)} \\
\multirow{13}{*}{}& \multirow{13}{*}{}& $L_0{:}~  \mathcal{T}(\widehat{\mb{T}}|Z_0(\widehat{\mb{T}}),S_L(\widehat{\mb{\Omega}}^{\text{S}}))$& 6.689(1.562)/2.803(0.782) & 3.785(0.882)/1.752(0.486) & 12.602(0.603)/4.954(0.327) & 1.761(0.531)/0.946(0.299) \\
\multirow{13}{*}{}& \multirow{13}{*}{}& $L_0{:}~  \mathcal{T}(\widehat{\mb{T}}|Z_0(\widehat{\mb{T}}),S_L(\widehat{\mb{T}}))$& 6.308(1.655)/3.158(0.784) & 3.621(0.961)/1.850(0.477) & 11.711(0.645)/5.638(0.347) & 1.872(0.524)/0.959(0.293) \\
\multirow{13}{*}{}& \multirow{13}{*}{}& $L_0{:}~  \widehat{\mb{T}}$& 49.675(4.593)/33.122(2.186) & 7.023(0.846)/4.662(0.335) & 30.694(0.570)/20.440(0.264) & 1.720(0.562)/0.956(0.286) \\
\bottomrule
\end{tabular}}
\caption*{Table~\ref{tab: norm loss, other Gaussian} (continued)}
\end{center}
 \end{table}

 \begin{table}[!htb]
 \renewcommand\arraystretch{1.3}
 \begin{center}
\resizebox{\textwidth}{!}{
\begin{tabular}{cclcccc}
\toprule
Sub-Gaussian &  &  & $L_1$ norm& Spectral norm &Frobenius norm& Max norm  \\
Graph&   $p$& Method & $n=200/n=400$& $n=200/n=400$ &$n=200/n=400$& $n=200/n=400$  \\
\midrule

\multirow{26.5}{*}{Band}&\multirow{13}{*}{200} &
$L_0{:}~ \widehat{\mb{\Omega}}^{\text{S}}$& 1.102(0.204)/0.575(0.091)  & 0.776(0.139)/{\bf 0.408(0.075)} & 3.063(0.166)/1.670(0.087) & 0.452(0.091)/{\bf 0.219(0.039)}  \\
\multirow{13}{*}{}& \multirow{13}{*}{}& $L_0{:}~ \mathcal{T}(\widehat{\mb{\Omega}}^{\text{S}}|Z_0(\widehat{\mb{\Omega}}^\text{US}),S_L(\widehat{\mb{\Omega}}^\text{S}))$&{\bf 1.095(0.206)}/{\bf 0.575(0.089)}  &{\bf 0.775(0.138)}/0.409(0.074) &{\bf 3.057(0.165)/1.664(0.088)} & 0.452(0.091)/0.222(0.042)  \\
\multirow{13}{*}{}& \multirow{13}{*}{}& $L_0{:}~  \mathcal{T}(\widehat{\mb{\Omega}}^{\text{S}}|Z_0(\widehat{\mb{T}}),S_L(\widehat{\mb{\Omega}}^{\text{S}}))$& 1.100(0.198)/0.578(0.089)  & 0.776(0.138)/0.411(0.074) & 3.064(0.164)/1.669(0.090) & 0.452(0.091)/0.230(0.047) \\
\multirow{13}{*}{}& \multirow{13}{*}{}& $L_0{:}~  \mathcal{T}(\widehat{\mb{T}}|Z_0(\widehat{\mb{T}}),S_L(\widehat{\mb{\Omega}}^{\text{S}}))$& 1.578(0.325)/0.754(0.119) & 1.056(0.190)/0.524(0.087) & 3.777(0.221)/1.969(0.121) & 0.461(0.088)/0.244(0.040) \\
\multirow{13}{*}{}& \multirow{13}{*}{}& $L_0{:}~  \mathcal{T}(\widehat{\mb{T}}|Z_0(\widehat{\mb{T}}),S_L(\widehat{\mb{T}}))$& 1.928(0.390)/0.941(0.193) & 1.133(0.188)/0.570(0.088) & 4.421(0.246)/2.339(0.140) & 0.462(0.087)/0.268(0.038) \\
\multirow{13}{*}{}& \multirow{13}{*}{}& $L_0{:}~  \widehat{\mb{T}}$& 16.316(0.827)/10.235(0.375) & 3.340(0.159)/2.220(0.099) & 15.511(0.236)/10.358(0.113) & 0.462(0.087)/0.237(0.027) \\
\bottomrule
\end{tabular}}
\caption{Average (standard deviation) of each matrix norm loss over 100 simulation replications under sub-Gaussian graph settings.}
\label{tab: norm loss, subGaussian}
\end{center}
 \end{table}

 \begin{table}[!htb]
 \renewcommand\arraystretch{1.3}
 \begin{center}
\resizebox{\textwidth}{!}{
\begin{tabular}{cclcccc}
\toprule
Sub-Gaussian &  &  & $L_1$ norm& Spectral norm &Frobenius norm& Max norm  \\
Graph&   $p$& Method & $n=200/n=400$& $n=200/n=400$ &$n=200/n=400$& $n=200/n=400$  \\
\midrule

\multirow{26.5}{*}{Hub}&\multirow{13}{*}{200} &
$L_0{:}~ \widehat{\mb{\Omega}}^{\text{S}}$& 7.008(1.160)/{\bf 4.063(0.533)}  & 4.804(1.158)/{\bf 2.619(0.549)} &{\bf 12.094(1.272)}/{\bf 6.906(0.559)} & 4.499(1.269)/2.442(0.612)  \\
\multirow{13}{*}{}& \multirow{13}{*}{}& $L_0{:}~ \mathcal{T}(\widehat{\mb{\Omega}}^{\text{S}}|Z_0(\widehat{\mb{\Omega}}^\text{US}),S_L(\widehat{\mb{\Omega}}^\text{S}))$& 7.084(1.169)/4.174(0.622)  & 4.844(1.148)/2.682(0.557) & 12.306(1.246)/7.031(0.594) & 4.499(1.269)/2.464(0.591)  \\
\multirow{13}{*}{}& \multirow{13}{*}{}& $L_0{:}~  \mathcal{T}(\widehat{\mb{\Omega}}^{\text{S}}|Z_0(\widehat{\mb{T}}),S_L(\widehat{\mb{\Omega}}^{\text{S}}))$& 7.045(1.182)/4.156(0.594)  & 4.827(1.157)/2.672(0.552) & 12.261(1.263)/7.013(0.588) & 4.499(1.269)/2.461(0.595) \\
\multirow{13}{*}{}& \multirow{13}{*}{}& $L_0{:}~  \mathcal{T}(\widehat{\mb{T}}|Z_0(\widehat{\mb{T}}),S_L(\widehat{\mb{\Omega}}^{\text{S}}))$& 11.775(2.445)/6.456(1.178) & 5.964(1.354)/3.277(0.663) & 15.344(1.548)/8.809(0.738) & 4.501(1.266)/2.461(0.595) \\
\multirow{13}{*}{}& \multirow{13}{*}{}& $L_0{:}~  \mathcal{T}(\widehat{\mb{T}}|Z_0(\widehat{\mb{T}}),S_L(\widehat{\mb{T}}))$& 13.481(2.965)/6.925(1.418) & 6.459(1.386)/3.451(0.687) & 18.278(1.486)/9.500(0.777) & 4.509(1.270)/2.508(0.598) \\
\multirow{13}{*}{}& \multirow{13}{*}{}& $L_0{:}~  \widehat{\mb{T}}$& 107.968(6.698)/70.427(3.220) & 15.730(0.957)/10.386(0.549) & 59.902(1.203)/40.680(0.596) & 4.609(1.158)/2.693(0.446) \\
\bottomrule
\end{tabular}}
\caption*{Table~\ref{tab: norm loss, subGaussian} (continued)}
\end{center}
 \end{table}

 \begin{table}[!htb]
 \renewcommand\arraystretch{1.3}
 \begin{center}
\resizebox{\textwidth}{!}{
\begin{tabular}{cclcccc}
\toprule
Gaussian &   &  & Precision& Sensitivity &Specificity& MCC  \\
Graph&   $p$& Method & $n=200/n=400$& $n=200/n=400$ &$n=200/n=400$& $n=200/n=400$  \\
\midrule

\multirow{22}{*}{Random}&\multirow{10.5}{*}{200} &
$L_0{:}~ \widehat{\mb{\Omega}}^{\text{S}}$& 0.935(0.014)/0.977(0.008)  & 0.707(0.023)/0.966(0.010) &{\bf 0.999(0.000)/1.000(0.000)} & 0.810(0.015)/0.971(0.007) \\
\multirow{10.5}{*}{}& \multirow{10.5}{*}{}& $L_0{:}~ \mathcal{T}(\widehat{\mb{\Omega}}^{\text{S}}|Z_0(\widehat{\mb{\Omega}}^\text{US}),S_L(\widehat{\mb{\Omega}}^\text{S}))$& 0.937(0.014)/{\bf 0.979(0.009)}  & 0.705(0.023)/0.965(0.010) &{\bf 0.999(0.000)/1.000(0.000)} & 0.810(0.016)/{\bf 0.972(0.007)}  \\
\multirow{10.5}{*}{}& \multirow{10.5}{*}{}& $L_0{:}~  \mathcal{T}(\widehat{\mb{\Omega}}^{\text{S}}|Z_0(\widehat{\mb{T}}),S_L(\widehat{\mb{\Omega}}^{\text{S}}))$& 0.935(0.014)/0.977(0.008)  & 0.704(0.024)/0.964(0.010) &{\bf 0.999(0.000)/1.000(0.000)} & 0.809(0.016)/0.970(0.007) \\
\multirow{10.5}{*}{}& \multirow{10.5}{*}{}& $L_0{:}~  \mathcal{T}(\widehat{\mb{T}}|Z_0(\widehat{\mb{T}}),S_L(\widehat{\mb{\Omega}}^{\text{S}}))$& 0.935(0.014)/0.977(0.008)  & 0.704(0.024)/0.964(0.010) & {\bf 0.999(0.000)/1.000(0.000)} & 0.809(0.016)/0.970(0.007) \\
\multirow{10.5}{*}{}& \multirow{10.5}{*}{}& $L_0{:}~  \mathcal{T}(\widehat{\mb{T}}|Z_0(\widehat{\mb{T}}),S_L(\widehat{\mb{T}}))$& 0.911(0.020)/0.940(0.013) & 0.550(0.026)/0.925(0.013) &{\bf 0.999(0.000)}/0.999(0.000) & 0.703(0.018)/0.931(0.011) \\
\multirow{10.5}{*}{}& \multirow{10.5}{*}{}& $L_1{:}~ \widehat{\mb{\Omega}}^{\text{S}}$& 0.925(0.017)/0.924(0.014) & 0.495(0.036)/0.947(0.014) &{\bf 0.999(0.000)}/0.998(0.000) & 0.672(0.027)/0.934(0.009) \\
\bottomrule
\end{tabular}}
\caption{Average (standard deviation) of each support-recovery metric over 100 simulation replications under Gaussian random, hub, and cluster graph settings.}
\label{tab: supp recover, other Gaussian}
\end{center}
 \end{table}

 \begin{table}[!htb]
 \renewcommand\arraystretch{1.3}
 \begin{center}
\resizebox{\textwidth}{!}{
\begin{tabular}{cclcccc}
\toprule
Gaussian &   &  & Precision& Sensitivity &Specificity& MCC  \\
Graph&   $p$& Method & $n=200/n=400$& $n=200/n=400$ &$n=200/n=400$& $n=200/n=400$  \\
\midrule

\multirow{22}{*}{Cluster}&\multirow{10.5}{*}{200} &
$L_0{:}~ \widehat{\mb{\Omega}}^{\text{S}}$& 0.960(0.010)/0.989(0.004)  & 0.775(0.028)/0.998(0.003) &{\bf 0.999(0.000)/1.000(0.000)} & 0.859(0.018)/0.993(0.003) \\
\multirow{10.5}{*}{}& \multirow{10.5}{*}{}& $L_0{:}~ \mathcal{T}(\widehat{\mb{\Omega}}^{\text{S}}|Z_0(\widehat{\mb{\Omega}}^\text{US}),S_L(\widehat{\mb{\Omega}}^\text{S}))$&{\bf 0.962(0.010)/0.991(0.004)}  & 0.775(0.028)/0.998(0.003) &{\bf 0.999(0.000)/1.000(0.000)} & 0.860(0.018)/{\bf 0.994(0.003)}  \\
\multirow{10.5}{*}{}& \multirow{10.5}{*}{}& $L_0{:}~  \mathcal{T}(\widehat{\mb{\Omega}}^{\text{S}}|Z_0(\widehat{\mb{T}}),S_L(\widehat{\mb{\Omega}}^{\text{S}}))$&{\bf 0.962(0.010)/0.991(0.004)}  & 0.775(0.028)/0.998(0.003) &{\bf 0.999(0.000)/1.000(0.000)} & 0.860(0.018)/{\bf 0.994(0.003)} \\
\multirow{10.5}{*}{}& \multirow{10.5}{*}{}& $L_0{:}~  \mathcal{T}(\widehat{\mb{T}}|Z_0(\widehat{\mb{T}}),S_L(\widehat{\mb{\Omega}}^{\text{S}}))$&{\bf 0.962(0.010)/0.991(0.004)}  & 0.775(0.028)/0.998(0.003) &{\bf 0.999(0.000)/1.000(0.000)} & 0.860(0.018)/{\bf 0.994(0.003)} \\
\multirow{10.5}{*}{}& \multirow{10.5}{*}{}& $L_0{:}~  \mathcal{T}(\widehat{\mb{T}}|Z_0(\widehat{\mb{T}}),S_L(\widehat{\mb{T}}))$& 0.921(0.011)/0.933(0.007) &{\bf 0.885(0.019)/0.998(0.002)} & 0.998(0.000)/0.998(0.000) &{\bf 0.900(0.012)}/0.964(0.004) \\
\multirow{10.5}{*}{}& \multirow{10.5}{*}{}& $L_1{:}~ \widehat{\mb{\Omega}}^{\text{S}}$& 0.801(0.017)/0.813(0.012) & 0.360(0.019)/0.722(0.020) & 0.997(0.000)/0.995(0.000) & 0.529(0.017)/0.760(0.012) \\
\bottomrule
\end{tabular}}
\caption*{Table~\ref{tab: supp recover, other Gaussian} (continued)}
\end{center}
 \end{table}

\begin{table}[!htb]
 \renewcommand\arraystretch{1.3}
 \begin{center}
\resizebox{\textwidth}{!}{
\begin{tabular}{cclcccc}
\toprule
Sub-Gaussian &   &  & Precision& Sensitivity &Specificity& MCC  \\
Graph&   $p$& Method & $n=200/n=400$& $n=200/n=400$ &$n=200/n=400$& $n=200/n=400$  \\
\midrule

\multirow{22}{*}{Band}&\multirow{10.5}{*}{200} &
$L_0{:}~ \widehat{\mb{\Omega}}^{\text{S}}$& 0.981(0.007)/0.992(0.005)  &{\bf 0.979(0.007)}/1.000(0.001) &{\bf 1.000(0.000)/1.000(0.000)} & 0.979(0.006)/0.996(0.002) \\
\multirow{10.5}{*}{}& \multirow{10.5}{*}{}& $L_0{:}~ \mathcal{T}(\widehat{\mb{\Omega}}^{\text{S}}|Z_0(\widehat{\mb{\Omega}}^\text{US}),S_L(\widehat{\mb{\Omega}}^\text{S}))$&{\bf 0.983(0.007)/0.994(0.004)}  & 0.978(0.007)/1.000(0.001) &{\bf 1.000(0.000)/1.000(0.000)} &{\bf 0.980(0.006)/0.997(0.002)}  \\
\multirow{10.5}{*}{}& \multirow{10.5}{*}{}& $L_0{:}~  \mathcal{T}(\widehat{\mb{\Omega}}^{\text{S}}|Z_0(\widehat{\mb{T}}),S_L(\widehat{\mb{\Omega}}^{\text{S}}))$& 0.982(0.006)/{\bf 0.994(0.004)}  & 0.978(0.008)/1.000(0.001) &{\bf 1.000(0.000)/1.000(0.000)} &{\bf 0.980(0.006)/0.997(0.002)} \\
\multirow{10.5}{*}{}& \multirow{10.5}{*}{}& $L_0{:}~  \mathcal{T}(\widehat{\mb{T}}|Z_0(\widehat{\mb{T}}),S_L(\widehat{\mb{\Omega}}^{\text{S}}))$& 0.982(0.006)/{\bf 0.994(0.004)}  & 0.978(0.008)/1.000(0.001) &{\bf 1.000(0.000)/1.000(0.000)} &{\bf 0.980(0.006)/0.997(0.002)} \\
\multirow{10.5}{*}{}& \multirow{10.5}{*}{}& $L_0{:}~  \mathcal{T}(\widehat{\mb{T}}|Z_0(\widehat{\mb{T}}),S_L(\widehat{\mb{T}}))$& 0.926(0.014)/0.931(0.009) & 0.950(0.011)/0.998(0.002) & 0.998(0.000)/0.998(0.000) & 0.937(0.009)/0.963(0.005) \\
\multirow{10.5}{*}{}& \multirow{10.5}{*}{}& $L_1{:}~ \widehat{\mb{\Omega}}^{\text{S}}$& 0.780(0.017)/0.809(0.016) & 0.928(0.017)/1.000(0.001) & 0.995(0.001)/0.995(0.000) & 0.848(0.012)/0.897(0.009) \\
\bottomrule
\end{tabular}}
\caption{Average (standard deviation) of each support-recovery metric over 100 simulation replications under sub-Gaussian settings.}
\label{tab: supp recover, subGaussian}
\end{center}
 \end{table}

 \begin{table}[!htb]
 \renewcommand\arraystretch{1.3}
 \begin{center}
\resizebox{\textwidth}{!}{
\begin{tabular}{cclcccc}
\toprule
Sub-Gaussian &   &  & Precision& Sensitivity &Specificity& MCC  \\
Graph&   $p$& Method & $n=200/n=400$& $n=200/n=400$ &$n=200/n=400$& $n=200/n=400$  \\
\midrule

\multirow{22}{*}{Hub}&\multirow{10.5}{*}{200} &
$L_0{:}~ \widehat{\mb{\Omega}}^{\text{S}}$& 0.984(0.008)/0.991(0.007)  & 0.988(0.008)/1.000(0.001) &{\bf 1.000(0.000)/1.000(0.000)} &{\bf 0.986(0.006)}/0.995(0.003) \\
\multirow{10.5}{*}{}& \multirow{10.5}{*}{}& $L_0{:}~ \mathcal{T}(\widehat{\mb{\Omega}}^{\text{S}}|Z_0(\widehat{\mb{\Omega}}^\text{US}),S_L(\widehat{\mb{\Omega}}^\text{S}))$& 0.985(0.008)/{\bf 0.995(0.006)}  & 0.983(0.008)/0.998(0.003) &{\bf 1.000(0.000)/1.000(0.000)} & 0.984(0.006)/{\bf 0.997(0.003)}  \\
\multirow{10.5}{*}{}& \multirow{10.5}{*}{}& $L_0{:}~  \mathcal{T}(\widehat{\mb{\Omega}}^{\text{S}}|Z_0(\widehat{\mb{T}}),S_L(\widehat{\mb{\Omega}}^{\text{S}}))$&{\bf 0.986(0.008)/0.995(0.006)}  & 0.984(0.008)/0.999(0.003) &{\bf 1.000(0.000)/1.000(0.000)} & 0.985(0.006)/{\bf 0.997(0.003)} \\
\multirow{10.5}{*}{}& \multirow{10.5}{*}{}& $L_0{:}~  \mathcal{T}(\widehat{\mb{T}}|Z_0(\widehat{\mb{T}}),S_L(\widehat{\mb{\Omega}}^{\text{S}}))$&{\bf 0.986(0.008)/0.995(0.006)}  & 0.984(0.008)/0.999(0.003) &{\bf 1.000(0.000)/1.000(0.000)} & 0.985(0.006)/{\bf 0.997(0.003)} \\
\multirow{10.5}{*}{}& \multirow{10.5}{*}{}& $L_0{:}~  \mathcal{T}(\widehat{\mb{T}}|Z_0(\widehat{\mb{T}}),S_L(\widehat{\mb{T}}))$& 0.933(0.016)/0.925(0.012) & 0.931(0.016)/0.999(0.002) & 0.999(0.000)/0.999(0.000) & 0.931(0.012)/0.961(0.006) \\
\multirow{10.5}{*}{}& \multirow{10.5}{*}{}& $L_1{:}~ \widehat{\mb{\Omega}}^{\text{S}}$& 0.801(0.024)/0.817(0.020) & 1.000(0.001)/{\bf 1.000(0.000)} & 0.998(0.000)/0.998(0.000) & 0.894(0.013)/0.903(0.011) \\
\bottomrule
\end{tabular}}
\caption*{Table~\ref{tab: supp recover, subGaussian} (continued)}
\end{center}
 \end{table}

 \begin{table}[!htb]
 \renewcommand\arraystretch{1.3}
 \begin{center}
\resizebox{\textwidth}{!}{
\begin{tabular}{cclccccc}
\toprule
Gaussian &   &  & $\TrueLength_{(i,j)\in S_{\mb{\Omega}}}$& $\AvgLength_{(i,j)\in S_{\mb{\Omega}}}$ &$\CovRate_{(i,j)\in S_{\mb{\Omega}}}$& $\AbsAvgZ_{(i,j)\in S_{\mb{\Omega}}}$ & $\SDZ_{(i,j)\in S_{\mb{\Omega}}}$  \\
Graph&   $p$& Method & $n=200/n=400$& $n=200/n=400$ &$n=200/n=400$& $n=200/n=400$ & $n=200/n=400$ \\
\midrule

\multirow{4.25}{*}{Band}& \multirow{2}{*}{200}& $L_0{:}~ \widehat{\mb{T}}$& 0.318(0.038)/0.225(0.027)  &{\bf  0.343(0.042)/0.232(0.028)}&{\bf  0.890(0.032)/0.938(0.024)} &{\bf 0.552(0.144)/0.282(0.112)} &{\bf  1.080(0.083)/1.012(0.073)}  \\
\multirow{4.25}{*}{}& \multirow{2}{*}{}& $L_1{:}~ \widehat{\mb{T}}$& 0.318(0.038)/0.225(0.027) & 0.244(0.035)/0.189(0.026)& 0.714(0.137)/0.847(0.088)& 1.193(0.537)/0.670(0.407)& 1.376(0.206)/1.144(0.117) \\ 
\cmidrule{2-8}

\multirow{4.25}{*}{}& \multirow{2}{*}{400}& $L_0{:}~ \widehat{\mb{T}}$& 0.318(0.038)/0.225(0.027)  &{\bf  0.346(0.042)/0.232(0.028)}&{\bf  0.864(0.040)/0.930(0.025)} &{\bf 0.658(0.172)/0.324(0.115) }&{\bf 1.131(0.092)/1.032(0.074)} \\
\multirow{4.25}{*}{}& \multirow{2}{*}{}& $L_1{:}~ \widehat{\mb{T}}$& 0.318(0.038)/0.225(0.027) & 0.229(0.033)/0.184(0.025)& 0.565(0.176)/0.788(0.129)& 1.827(0.743)/0.908(0.506)& 1.582(0.305)/1.195(0.131) \\ 
\cmidrule{1-8}

\multirow{4.25}{*}{Random}& \multirow{2}{*}{200}& $L_0{:}~ \widehat{\mb{T}}$& 0.522(0.224)/0.369(0.158)  &{\bf  0.551(0.243)/0.380(0.165)}&{\bf  0.892(0.034)/0.934(0.025)} &{\bf 0.350(0.142)/0.214(0.111)} &1.171(0.091)/1.038(0.075)  \\
\multirow{4.25}{*}{}& \multirow{2}{*}{}& $L_1{:}~ \widehat{\mb{T}}$& 0.522(0.224)/0.369(0.158) & 0.418(0.164)/0.323(0.125)& 0.868(0.155)/0.920(0.069)& 0.517(0.684)/0.349(0.327)&{\bf  1.135(0.268)/1.022(0.119)} \\
\bottomrule
\end{tabular}}
\caption{Average (standard deviation) of each asymptotic-normality metric over all entries in $S_{\mb{\Omega}}$ under Gaussian settings  based on 100 simulation replications.}
\label{tab: AN S_omega, Gaussian}
\end{center}
 \end{table}

 \begin{table}[!htb]
 \renewcommand\arraystretch{1.3}
 \begin{center}
\resizebox{\textwidth}{!}{
\begin{tabular}{cclccccc}
\toprule
Gaussian &   &  & $\TrueLength_{(i,j)\in S_{\mb{\Omega}}^c}$& $\AvgLength_{(i,j)\in S_{\mb{\Omega}}^c}$ &$\CovRate_{(i,j)\in S_{\mb{\Omega}}^c}$& $\AbsAvgZ_{(i,j)\in S_{\mb{\Omega}}^c}$ & $\SDZ_{(i,j)\in S_{\mb{\Omega}}^c}$  \\
Graph&   $p$& Method & $n=200/n=400$& $n=200/n=400$ &$n=200/n=400$& $n=200/n=400$ & $n=200/n=400$ \\
\midrule

\multirow{4.25}{*}{Band}& \multirow{2}{*}{200}& $L_0{:}~ \widehat{\mb{T}}$& 0.277(0.000)/0.196(0.000) &{\bf  0.298(0.002)/0.202(0.001)}&{\bf  0.949(0.022)/0.950(0.022)} &{\bf 0.080(0.060)/0.080(0.061)} & {\bf 1.004(0.072)/0.998(0.071)}  \\
\multirow{4.25}{*}{}& \multirow{2}{*}{}& $L_1{:}~ \widehat{\mb{T}}$& 0.277(0.000)/0.196(0.000) & 0.221(0.002)/0.170(0.001)& 0.962(0.033)/0.963(0.023)& 0.097(0.160)/0.090(0.119)& 0.926(0.066)/0.930(0.066) \\ 
\cmidrule{2-8}

\multirow{4.25}{*}{}& \multirow{2}{*}{400}& $L_0{:}~ \widehat{\mb{T}}$& 0.277(0.000)/0.196(0.000)  &{\bf  0.301(0.003)/0.202(0.001)}&{\bf  0.949(0.022)/0.950(0.022)} &{\bf 0.080(0.061)/0.080(0.060) }&{\bf 1.003(0.071)/0.998(0.071)} \\
\multirow{4.25}{*}{}& \multirow{2}{*}{}& $L_1{:}~ \widehat{\mb{T}}$& 0.277(0.000)/0.196(0.000) & 0.208(0.002)/0.166(0.001)& 0.962(0.035)/0.964(0.023)& 0.090(0.147)/0.084(0.108)& 0.933(0.066)/0.929(0.066) \\ 
\cmidrule{1-8}

\multirow{4.25}{*}{Random}& \multirow{2}{*}{200}& $L_0{:}~ \widehat{\mb{T}}$& 0.418(0.107)/0.295(0.076)  &{\bf  0.439(0.115)/0.303(0.079)}&{\bf  0.949(0.022)/0.950(0.022)} &{\bf 0.080(0.061)/0.080(0.060) }&{\bf 1.004(0.073)/0.999(0.071)}  \\
\multirow{4.25}{*}{}& \multirow{2}{*}{}& $L_1{:}~ \widehat{\mb{T}}$& 0.418(0.107)/0.295(0.076) & 0.349(0.076)/0.267(0.059)& 0.952(0.026)/0.958(0.021)& 0.196(0.190)/0.136(0.121)& 0.954(0.075)/0.948(0.072) \\
\bottomrule
\end{tabular}}
\caption{Average (standard deviation) of each asymptotic-normality metric over all entries in $S_{\mb{\Omega}}^c$ under Gaussian settings  based on 100 simulation replications.}
\label{tab: AN S_omega complement, Gaussian}
\end{center}
 \end{table}

 \begin{table}[!htb]
 \renewcommand\arraystretch{1.3}
 \begin{center}
\resizebox{\textwidth}{!}{
\begin{tabular}{cclccccc}
\toprule
Sub-Gaussian &   &  & $\TrueLength_{(i,j)\in\widehat{S}_{\mb{\Omega}}}$& $\AvgLength_{(i,j)\in\widehat{S}_{\mb{\Omega}}}$ &$\CovRate_{(i,j)\in\widehat{S}_{\mb{\Omega}}}$& $\AbsAvgZ_{(i,j)\in\widehat{S}_{\mb{\Omega}}}$ & $\SDZ_{(i,j)\in\widehat{S}_{\mb{\Omega}}}$  \\
Graph&   $p$& Method & $n=200/n=400$& $n=200/n=400$ &$n=200/n=400$& $n=200/n=400$ & $n=200/n=400$ \\
\midrule

\multirow{6.5}{*}{Band}&\multirow{3}{*}{200} &
$L_0{:}~ \widehat{\mb{\Omega}}^{\text{US}}$&{\bf  0.280(0.013)/0.188(0.016)}  & 0.306(0.021)/0.194(0.018) & {\bf 0.895(0.037)/0.937(0.025)} &{\bf  0.479(0.204)/0.230(0.142)} &{\bf  1.101(0.094)/1.025(0.073)} \\
\multirow{6.5}{*}{}& \multirow{3}{*}{}& $L_0{:}~ \widehat{\mb{T}}$& 0.295(0.003)/0.205(0.004)  &{\bf 0.318(0.011)/0.211(0.006)} & 0.861(0.036)/0.927(0.026) & 0.698(0.122)/0.324(0.126) & 1.116(0.088)/1.037(0.073)  \\
\multirow{6.5}{*}{}& \multirow{3}{*}{}& $L_1{:}~ \widehat{\mb{T}}$& 0.295(0.003)/0.205(0.004) & 0.190(0.010)/0.150(0.008) & 0.497(0.180)/0.756(0.184) & 2.138(1.006)/0.914(0.784) & 1.793(0.427)/1.367(0.278) \\ 
\cmidrule{2-8}

\multirow{6.5}{*}{}&\multirow{3}{*}{400}& $L_0{:}~ \widehat{\mb{\Omega}}^{\text{US}}$&{\bf  0.281(0.011)/0.188(0.016)}  & 0.311(0.021)/{\bf 0.194(0.018)} &{\bf  0.871(0.043)/0.936(0.026)} & {\bf 0.558(0.229)/0.244(0.143)} &{\bf  1.156(0.097)/1.024(0.072)}  \\
\multirow{6.5}{*}{}& \multirow{3}{*}{}& $L_0{:}~ \widehat{\mb{T}}$& 0.295(0.002)/0.205(0.004)  &{\bf  0.322(0.013)}/0.212(0.006) & 0.823(0.038)/0.925(0.027) & 0.845(0.128)/0.353(0.123) & 1.186(0.094)/1.037(0.073) \\
\multirow{6.5}{*}{}& \multirow{3}{*}{}& $L_1{:}~ \widehat{\mb{T}}$& 0.295(0.002)/0.205(0.004) & 0.181(0.009)/0.145(0.008) & 0.328(0.164)/0.690(0.223) & 3.139(1.292)/1.289(1.025) & 2.015(0.494)/1.419(0.327) \\
\bottomrule
\end{tabular}}
\caption{Average (standard deviation) of each asymptotic-normality metric over all entries in $\widehat{S}_{\mb{\Omega}}$ under sub-Gaussian settings  based on 100 simulation replications.}
\label{tab: AN S_omega hat, subGaussian}
\end{center}
 \end{table}

 \begin{table}[!htb]
 \renewcommand\arraystretch{1.3}
 \begin{center}
\resizebox{\textwidth}{!}{
\begin{tabular}{cclccccc}
\toprule
Sub-Gaussian &   &  & $\TrueLength_{(i,j)\in S_{\mb{\Omega}}}$& $\AvgLength_{(i,j)\in S_{\mb{\Omega}}}$ &$\CovRate_{(i,j)\in S_{\mb{\Omega}}}$& $\AbsAvgZ_{(i,j)\in S_{\mb{\Omega}}}$ & $\SDZ_{(i,j)\in S_{\mb{\Omega}}}$  \\
Graph&   $p$& Method & $n=200/n=400$& $n=200/n=400$ &$n=200/n=400$& $n=200/n=400$ & $n=200/n=400$ \\
\midrule

\multirow{4.25}{*}{Band}& \multirow{2}{*}{200}& $L_0{:}~ \widehat{\mb{T}}$& 0.290(0.006)/0.205(0.004)  &{\bf 0.312(0.012)/0.211(0.006) }&{\bf  0.873(0.038)/0.927(0.026)} &{\bf  0.605(0.169)/0.323(0.125)} &{\bf  1.119(0.089)/1.037(0.073)}  \\
\multirow{4.25}{*}{}& \multirow{2}{*}{}& $L_1{:}~ \widehat{\mb{T}}$& 0.290(0.006)/0.205(0.004) & 0.195(0.011)/0.150(0.008) & 0.622(0.218)/0.757(0.184) & 1.585(1.079)/0.909(0.783) & 1.589(0.434)/1.365(0.277) \\ 
\cmidrule{2-8}

\multirow{4.25}{*}{}& \multirow{2}{*}{400}& $L_0{:}~ \widehat{\mb{T}}$& 0.290(0.006)/0.205(0.004)  &{\bf  0.315(0.013)/0.212(0.006)} &{\bf  0.841(0.044)/0.925(0.027) }&{\bf  0.731(0.190)/0.351(0.123) }& {\bf 1.180(0.092)/1.038(0.073)} \\
\multirow{4.25}{*}{}& \multirow{2}{*}{}& $L_1{:}~ \widehat{\mb{T}}$& 0.290(0.006)/0.205(0.004) & 0.187(0.010)/0.145(0.008) & 0.488(0.240)/0.693(0.223) & 2.321(1.439)/1.277(1.023) & 1.737(0.522)/1.415(0.326) \\ 
\cmidrule{1-8}

\multirow{4.25}{*}{Random}& \multirow{2}{*}{200}& $L_0{:}~ \widehat{\mb{T}}$& 0.477(0.200)/0.337(0.142)  &{\bf  0.505(0.217)/0.347(0.147)} &{\bf  0.883(0.035)/0.933(0.026)} &{\bf  0.369(0.153)/0.232(0.134) }&{\bf  1.195(0.095)/1.037(0.076)}  \\
\multirow{4.25}{*}{}& \multirow{2}{*}{}& $L_1{:}~ \widehat{\mb{T}}$& 0.477(0.200)/0.337(0.142) & 0.352(0.095)/0.266(0.082) & 0.812(0.230)/0.867(0.146) & 0.747(1.135)/0.508(0.621) & 1.274(0.486)/1.144(0.272) \\
\bottomrule
\end{tabular}}
\caption{Average (standard deviation) of each asymptotic-normality metric over all entries in $S_{\mb{\Omega}}$ under sub-Gaussian settings  based on 100 simulation replications.}
\label{tab: AN S_omega, subGaussian}
\end{center}
 \end{table}

 \begin{table}[!htb]
 \renewcommand\arraystretch{1.3}
 \begin{center}
\resizebox{\textwidth}{!}{
\begin{tabular}{cclccccc}
\toprule
Sub-Gaussian &   &  & $\TrueLength_{(i,j)\in S_{\mb{\Omega}}^c}$& $\AvgLength_{(i,j)\in S_{\mb{\Omega}}^c}$ &$\CovRate_{(i,j)\in S_{\mb{\Omega}}^c}$& $\AbsAvgZ_{(i,j)\in S_{\mb{\Omega}}^c}$ & $\SDZ_{(i,j)\in S_{\mb{\Omega}}^c}$  \\
Graph&   $p$& Method & $n=200/n=400$& $n=200/n=400$ &$n=200/n=400$& $n=200/n=400$ & $n=200/n=400$ \\
\midrule

\multirow{4.25}{*}{Band}& \multirow{2}{*}{200}& $L_0{:}~ \widehat{\mb{T}}$& 0.277(0.000)/0.196(0.000) &{\bf  0.297(0.002)/0.201(0.001)} &{\bf  0.949(0.022)/0.950(0.022)} &{\bf  0.081(0.061)/0.081(0.061)} & {\bf 1.005(0.072)/1.000(0.071)} \\
\multirow{4.25}{*}{}& \multirow{2}{*}{}& $L_1{:}~ \widehat{\mb{T}}$& 0.277(0.000)/0.196(0.000) & 0.205(0.002)/0.159(0.001) & 0.948(0.037)/0.949(0.026) & 0.105(0.171)/0.096(0.126) & 0.994(0.071)/0.995(0.070) \\ 
\cmidrule{2-8}

\multirow{4.25}{*}{}& \multirow{2}{*}{400}& $L_0{:}~ \widehat{\mb{T}}$& 0.277(0.000)/0.196(0.000)  &{\bf  0.300(0.003)/0.202(0.001) }&{\bf  0.950(0.022)/0.950(0.022)} &{\bf  0.080(0.060)/0.080(0.060) }& {\bf 1.002(0.072)/0.999(0.071)} \\
\multirow{4.25}{*}{}& \multirow{2}{*}{}& $L_1{:}~ \widehat{\mb{T}}$& 0.277(0.000)/0.196(0.000) & 0.194(0.002)/0.154(0.001) & 0.948(0.038)/0.949(0.027) & 0.096(0.156)/0.090(0.116) & 0.995(0.070)/0.996(0.070) \\ 
\cmidrule{1-8}

\multirow{4.25}{*}{Random}& \multirow{2}{*}{200}& $L_0{:}~ \widehat{\mb{T}}$& 0.418(0.107)/0.295(0.076)  &{\bf  0.436(0.114)/0.302(0.079)} &{\bf  0.949(0.022)/0.950(0.022)} &{\bf  0.080(0.061)/0.080(0.060) }&{\bf  1.006(0.074)/1.000(0.071)}  \\
\multirow{4.25}{*}{}& \multirow{2}{*}{}& $L_1{:}~ \widehat{\mb{T}}$& 0.418(0.107)/0.295(0.076) & 0.331(0.058)/0.252(0.048) & 0.943(0.029)/0.948(0.023) & 0.206(0.204)/0.143(0.130) & 0.993(0.071)/0.995(0.071) \\
\bottomrule
\end{tabular}}
\caption{Average (standard deviation) of each asymptotic-normality metric over all entries in $S_{\mb{\Omega}}^c$ under sub-Gaussian settings  based on 100 simulation replications.}
\label{tab: AN S_omega complement, subGaussian}
\end{center}
 \end{table}

%%%%%%%%%%%%%%%%%%%normality plots
%Gaussian band
  \begin{sidewaysfigure}[th!]
  \caption*{$n=200, p=200$}
  \vspace{-0.43cm}
 \begin{minipage}{0.3\linewidth}
    \begin{minipage}{0.24\linewidth}
        \centering
        \includegraphics[width=\textwidth]{images/unbias_L0_Gaussian_Band_n200_p200_12.pdf}
    \end{minipage}
    \begin{minipage}{0.24\linewidth}
        \centering
        \includegraphics[width=\textwidth]{images/unbias_L0_Gaussian_Band_n200_p200_23.pdf}
    \end{minipage}
    \begin{minipage}{0.24\linewidth}
        \centering
        \includegraphics[width=\textwidth]{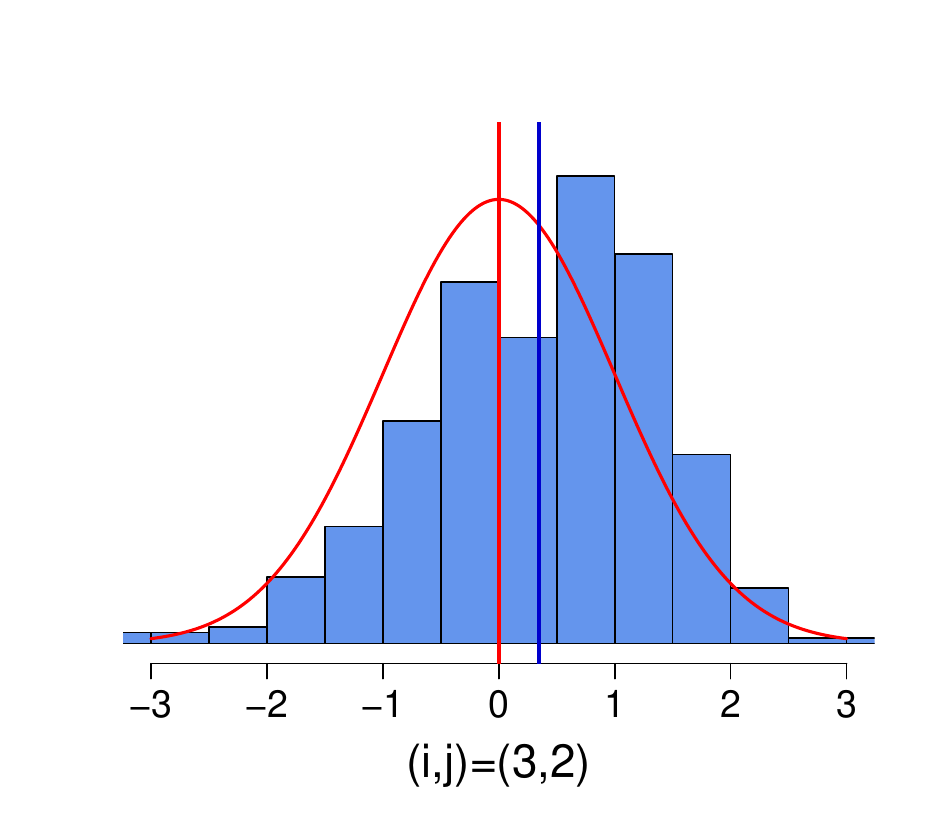}
    \end{minipage}
    \begin{minipage}{0.24\linewidth}
        \centering
        \includegraphics[width=\textwidth]{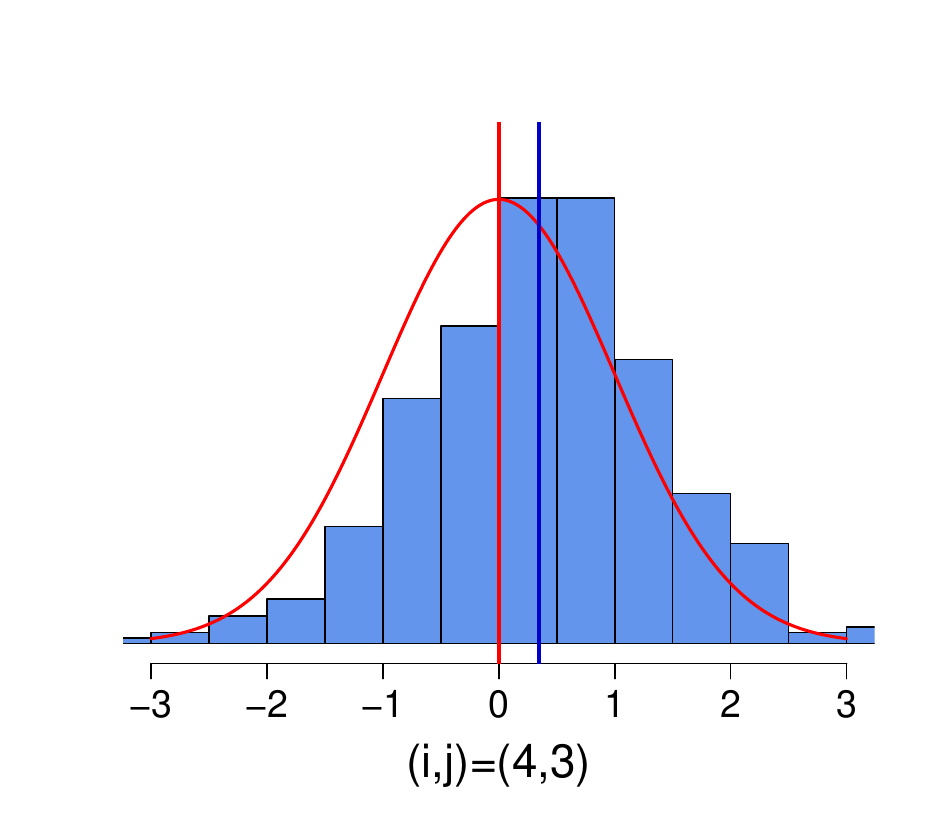}
    \end{minipage}
 \end{minipage}
 \hspace{1cm}
 \begin{minipage}{0.3\linewidth}
    \begin{minipage}{0.24\linewidth}
        \centering
        \includegraphics[width=\textwidth]{images/debias_L0_Gaussian_Band_n200_p200_12.pdf}
    \end{minipage}
    \begin{minipage}{0.24\linewidth}
        \centering
        \includegraphics[width=\textwidth]{images/debias_L0_Gaussian_Band_n200_p200_23.pdf}
    \end{minipage}
    \begin{minipage}{0.24\linewidth}
        \centering
        \includegraphics[width=\textwidth]{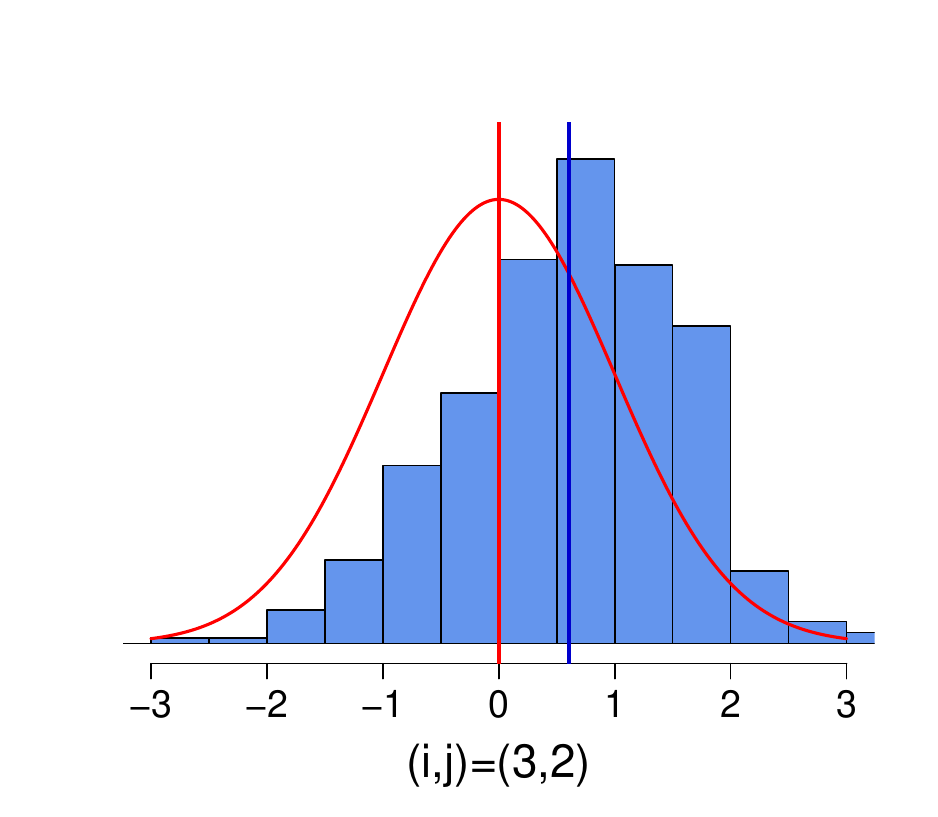}
    \end{minipage}
    \begin{minipage}{0.24\linewidth}
        \centering
        \includegraphics[width=\textwidth]{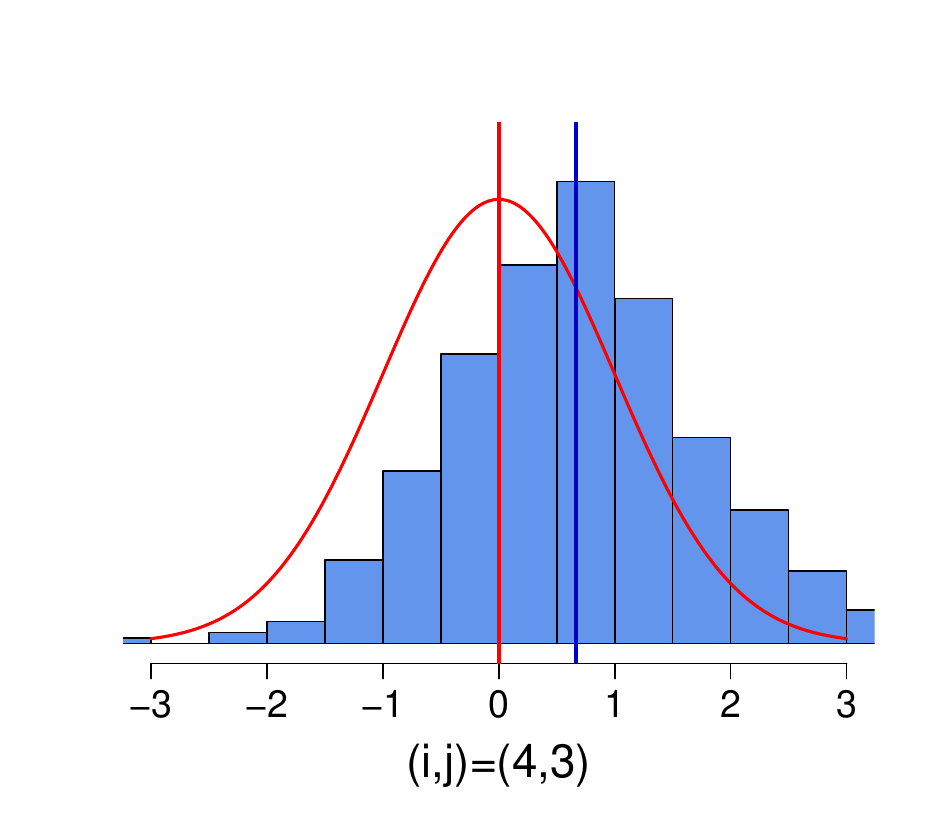}
    \end{minipage}    
 \end{minipage}
  \hspace{1cm}
 \begin{minipage}{0.3\linewidth}
     \begin{minipage}{0.24\linewidth}
        \centering
        \includegraphics[width=\textwidth]{images/L1_Gaussian_Band_n200_p200_12.pdf}
    \end{minipage}
    \begin{minipage}{0.24\linewidth}
        \centering
        \includegraphics[width=\textwidth]{images/L1_Gaussian_Band_n200_p200_23.pdf}
    \end{minipage}
    \begin{minipage}{0.24\linewidth}
        \centering
        \includegraphics[width=\textwidth]{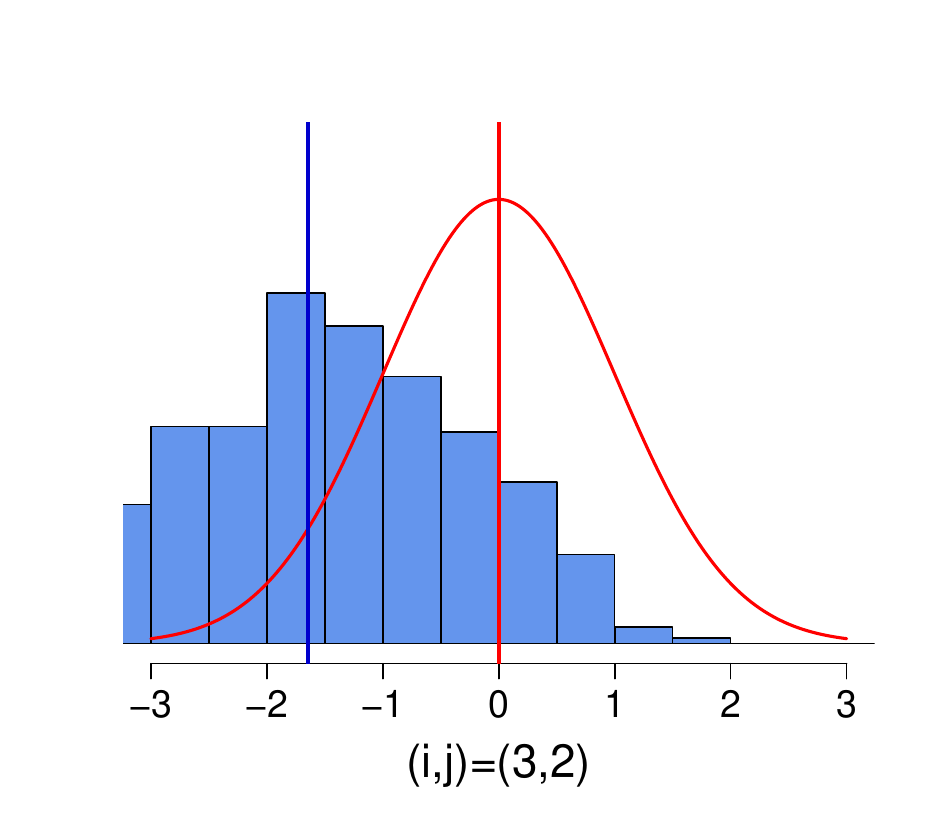}
    \end{minipage}
    \begin{minipage}{0.24\linewidth}
        \centering
        \includegraphics[width=\textwidth]{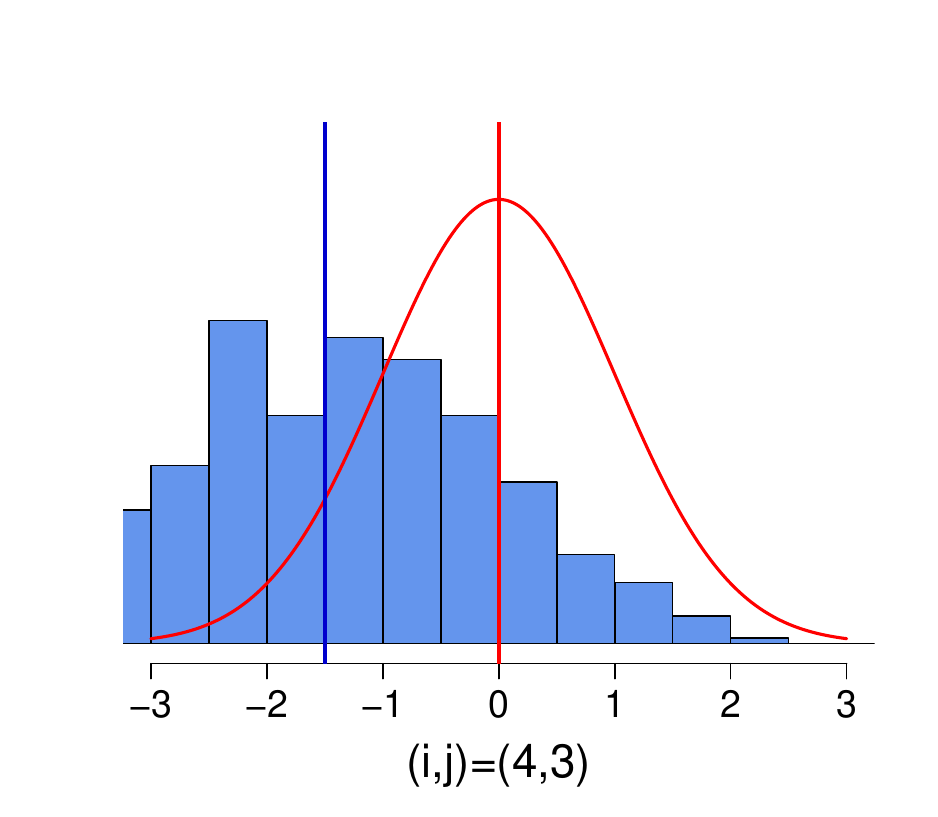}
    \end{minipage}
 \end{minipage}

  \caption*{$n=400, p=200$}
    \vspace{-0.43cm}
 \begin{minipage}{0.3\linewidth}
    \begin{minipage}{0.24\linewidth}
        \centering
        \includegraphics[width=\textwidth]{images/unbias_L0_Gaussian_Band_n400_p200_12.pdf}
    \end{minipage}
    \begin{minipage}{0.24\linewidth}
        \centering
        \includegraphics[width=\textwidth]{images/unbias_L0_Gaussian_Band_n400_p200_23.pdf}
    \end{minipage}
    \begin{minipage}{0.24\linewidth}
        \centering
        \includegraphics[width=\textwidth]{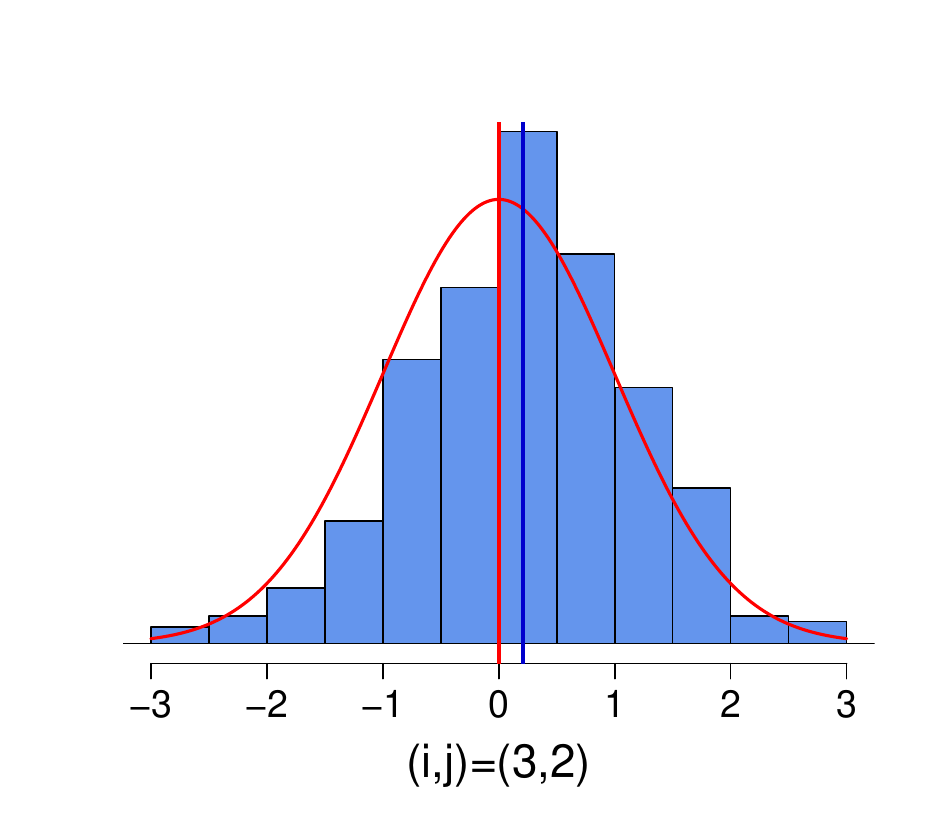}
    \end{minipage}
    \begin{minipage}{0.24\linewidth}
        \centering
        \includegraphics[width=\textwidth]{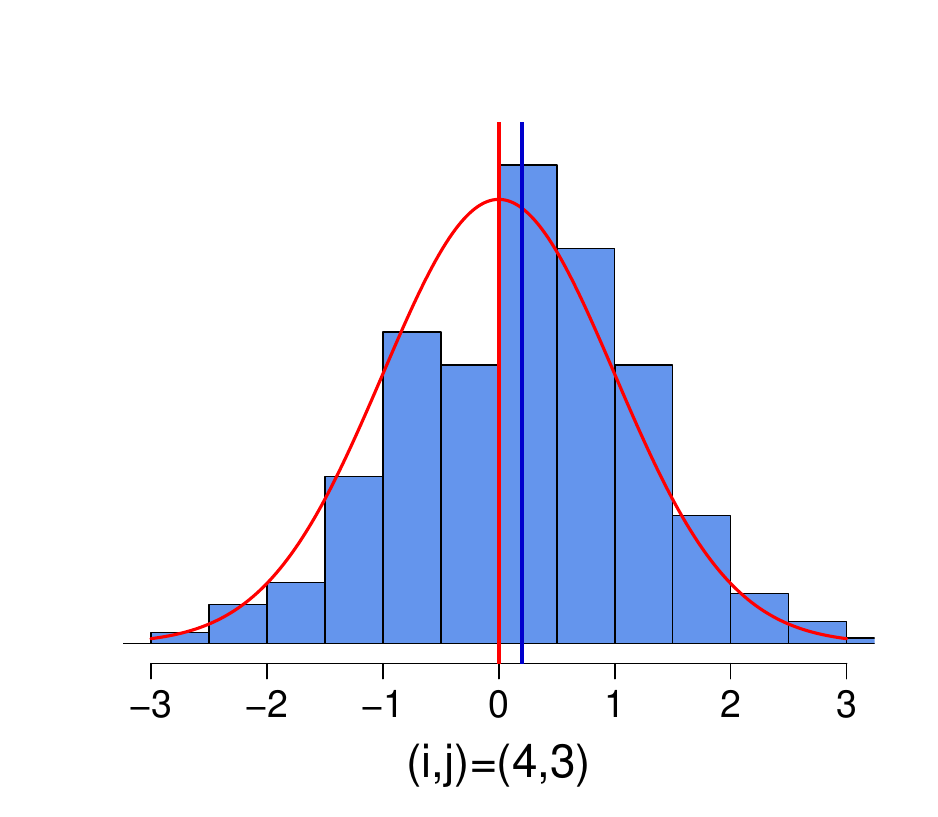}
    \end{minipage}
 \end{minipage}  
     \hspace{1cm}
 \begin{minipage}{0.3\linewidth}
    \begin{minipage}{0.24\linewidth}
        \centering
        \includegraphics[width=\textwidth]{images/debias_L0_Gaussian_Band_n400_p200_12.pdf}
    \end{minipage}
    \begin{minipage}{0.24\linewidth}
        \centering
        \includegraphics[width=\textwidth]{images/debias_L0_Gaussian_Band_n400_p200_23.pdf}
    \end{minipage}
    \begin{minipage}{0.24\linewidth}
        \centering
        \includegraphics[width=\textwidth]{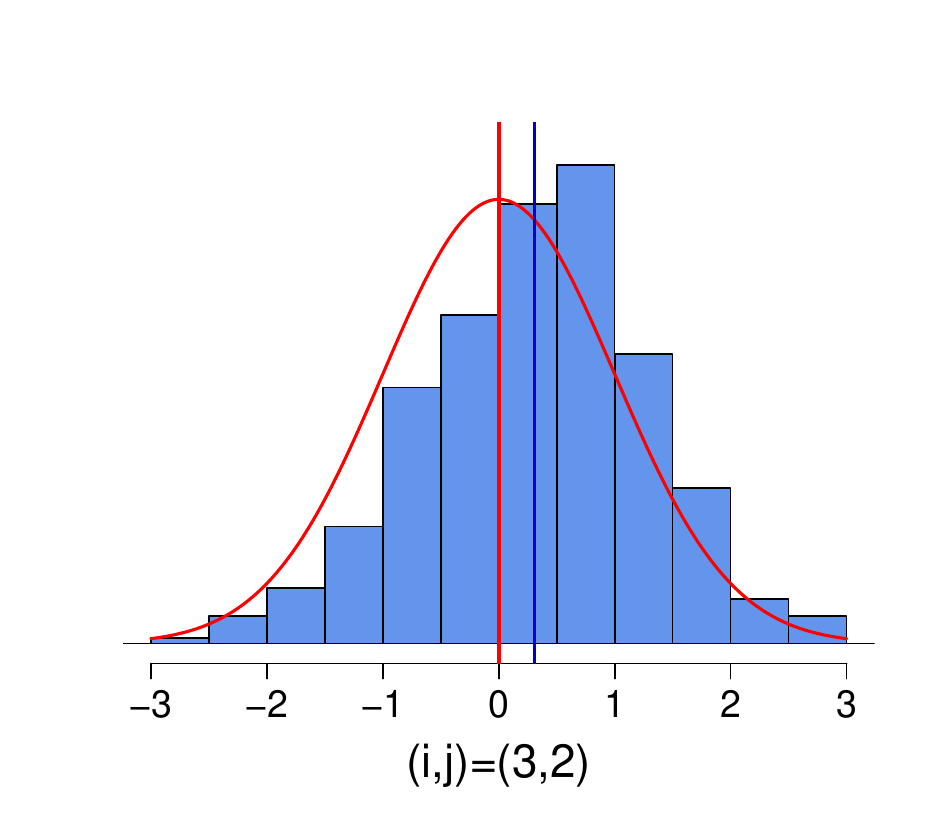}
    \end{minipage}
    \begin{minipage}{0.24\linewidth}
        \centering
        \includegraphics[width=\textwidth]{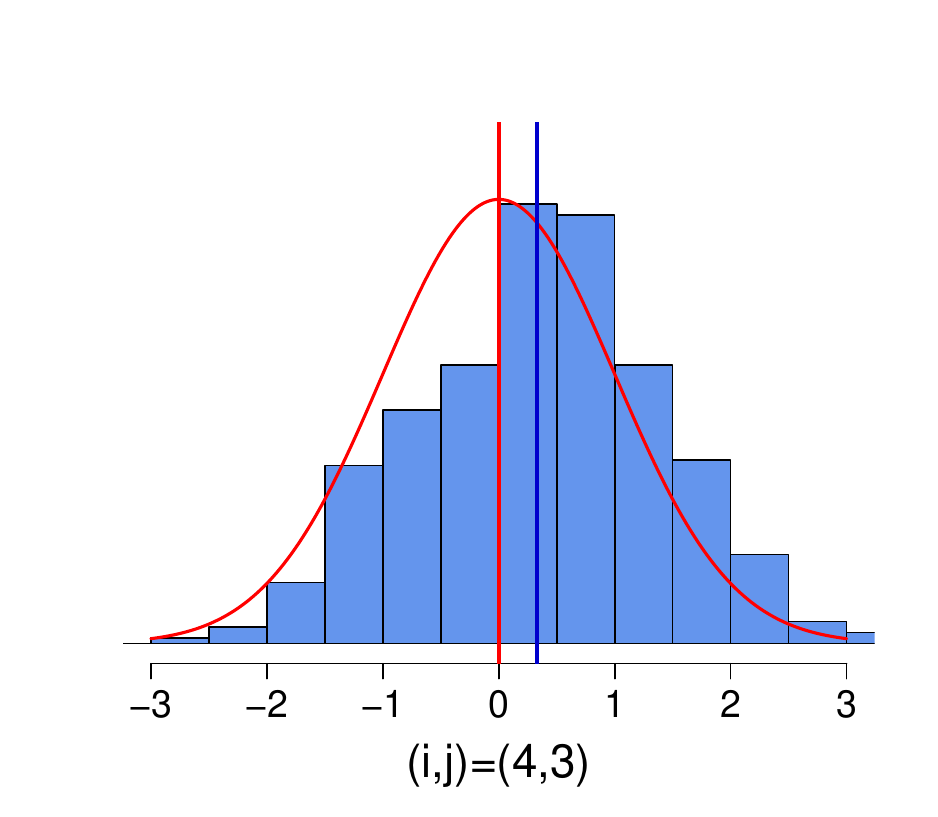}
    \end{minipage}
  \end{minipage}  
    \hspace{1cm}
 \begin{minipage}{0.3\linewidth}
    \begin{minipage}{0.24\linewidth}
        \centering
        \includegraphics[width=\textwidth]{images/L1_Gaussian_Band_n400_p200_12.pdf}
    \end{minipage}
    \begin{minipage}{0.24\linewidth}
        \centering
        \includegraphics[width=\textwidth]{images/L1_Gaussian_Band_n400_p200_23.pdf}
    \end{minipage}
    \begin{minipage}{0.24\linewidth}
        \centering
        \includegraphics[width=\textwidth]{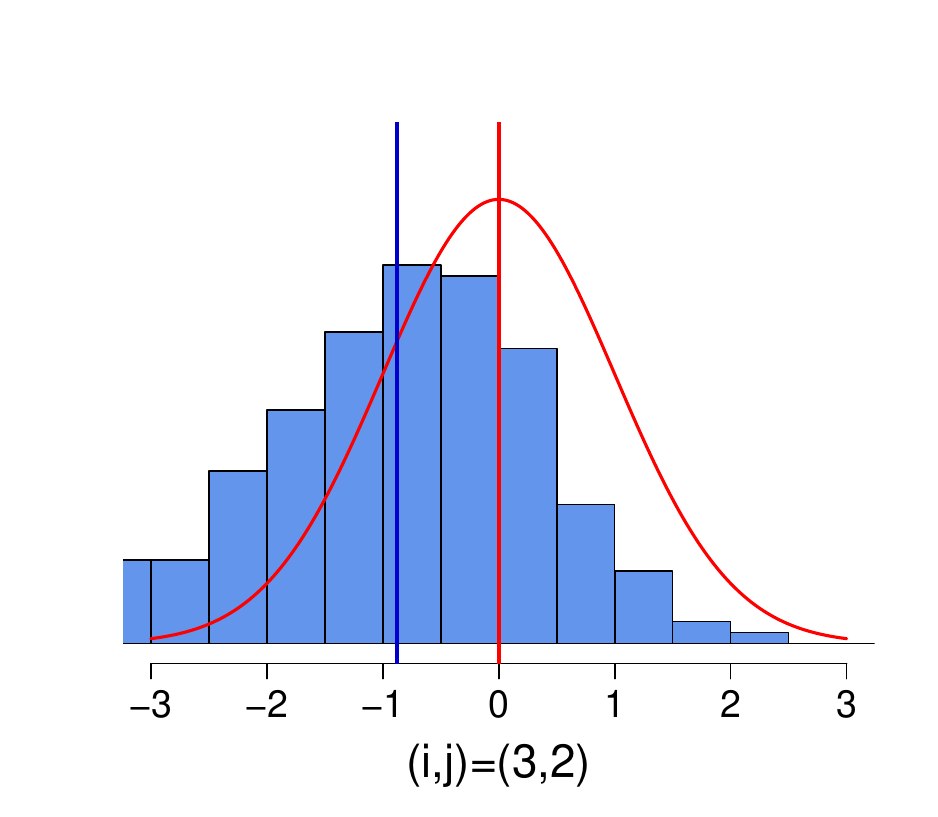}
    \end{minipage}
    \begin{minipage}{0.24\linewidth}
        \centering
        \includegraphics[width=\textwidth]{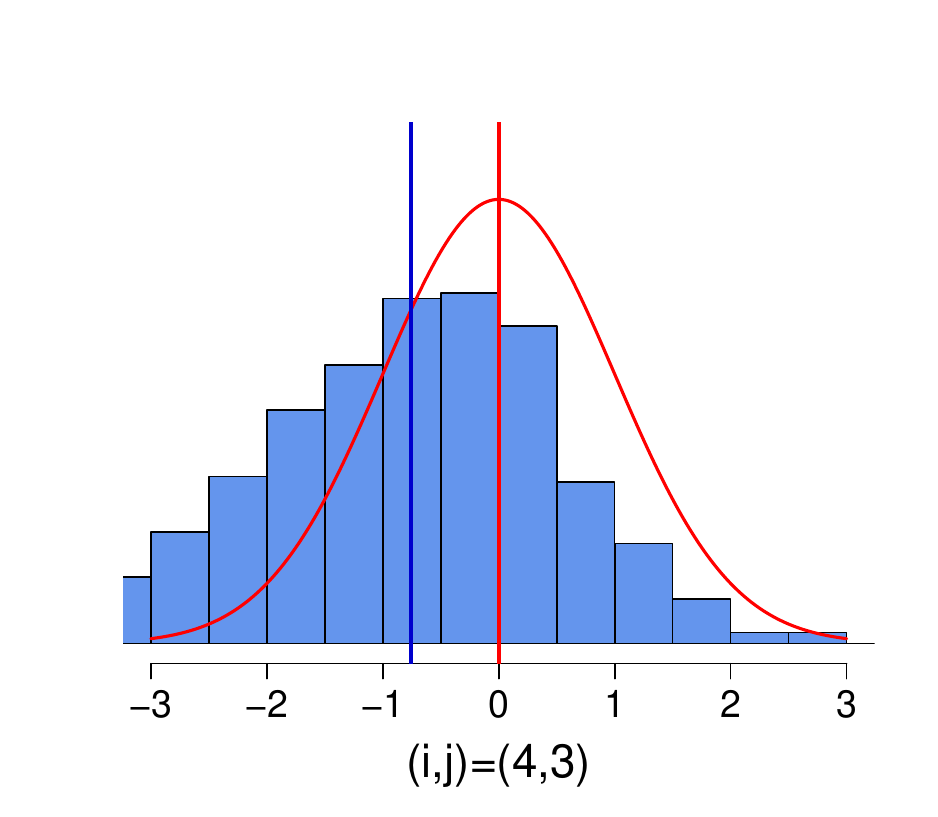}
    \end{minipage}
 \end{minipage}

  \caption*{$n=800, p=200$}
    \vspace{-0.43cm}
 \begin{minipage}{0.3\linewidth}
    \begin{minipage}{0.24\linewidth}
        \centering
        \includegraphics[width=\textwidth]{images/unbias_L0_Gaussian_Band_n800_p200_12.pdf}
    \end{minipage}
    \begin{minipage}{0.24\linewidth}
        \centering
        \includegraphics[width=\textwidth]{images/unbias_L0_Gaussian_Band_n800_p200_23.pdf}
    \end{minipage}
    \begin{minipage}{0.24\linewidth}
        \centering
        \includegraphics[width=\textwidth]{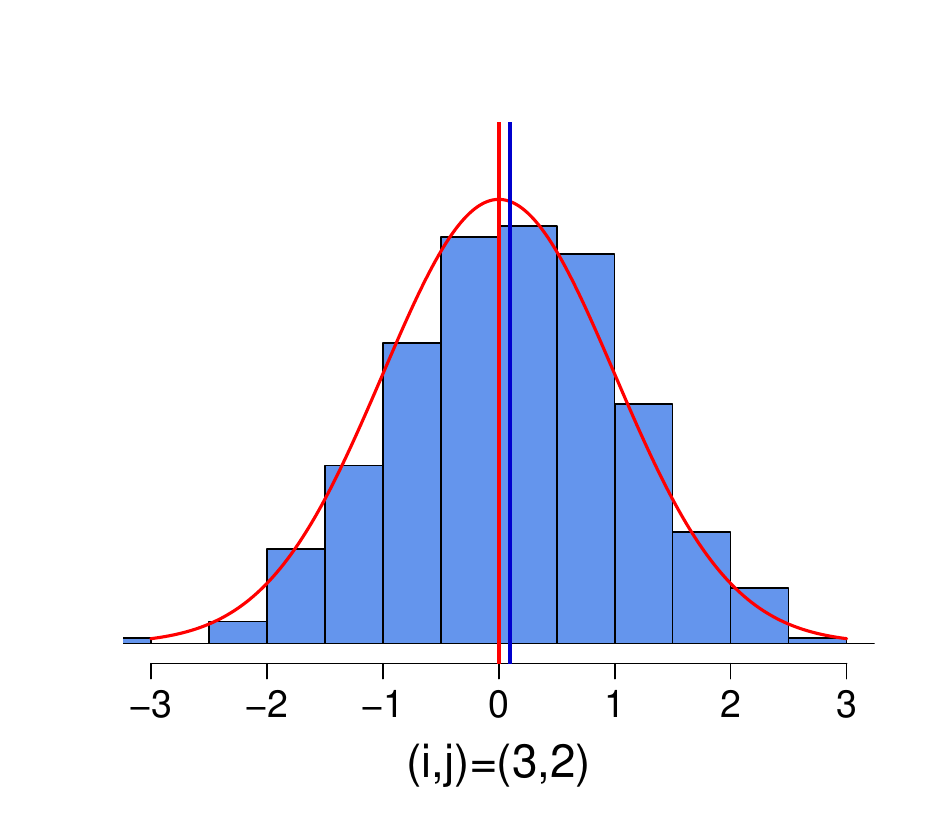}
    \end{minipage}
    \begin{minipage}{0.24\linewidth}
        \centering
        \includegraphics[width=\textwidth]{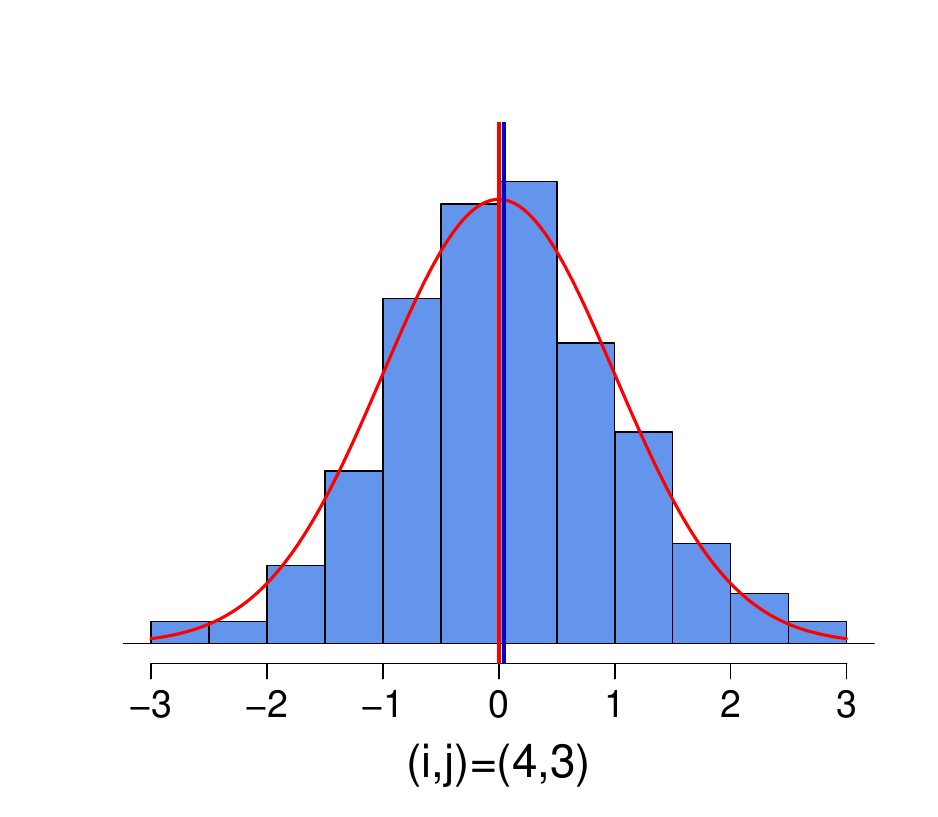}
    \end{minipage}
 \end{minipage} 
     \hspace{1cm}
 \begin{minipage}{0.3\linewidth}
    \begin{minipage}{0.24\linewidth}
        \centering
        \includegraphics[width=\textwidth]{images/debias_L0_Gaussian_Band_n800_p200_12.pdf}
    \end{minipage}
    \begin{minipage}{0.24\linewidth}
        \centering
        \includegraphics[width=\textwidth]{images/debias_L0_Gaussian_Band_n800_p200_23.pdf}
    \end{minipage}
    \begin{minipage}{0.24\linewidth}
        \centering
        \includegraphics[width=\textwidth]{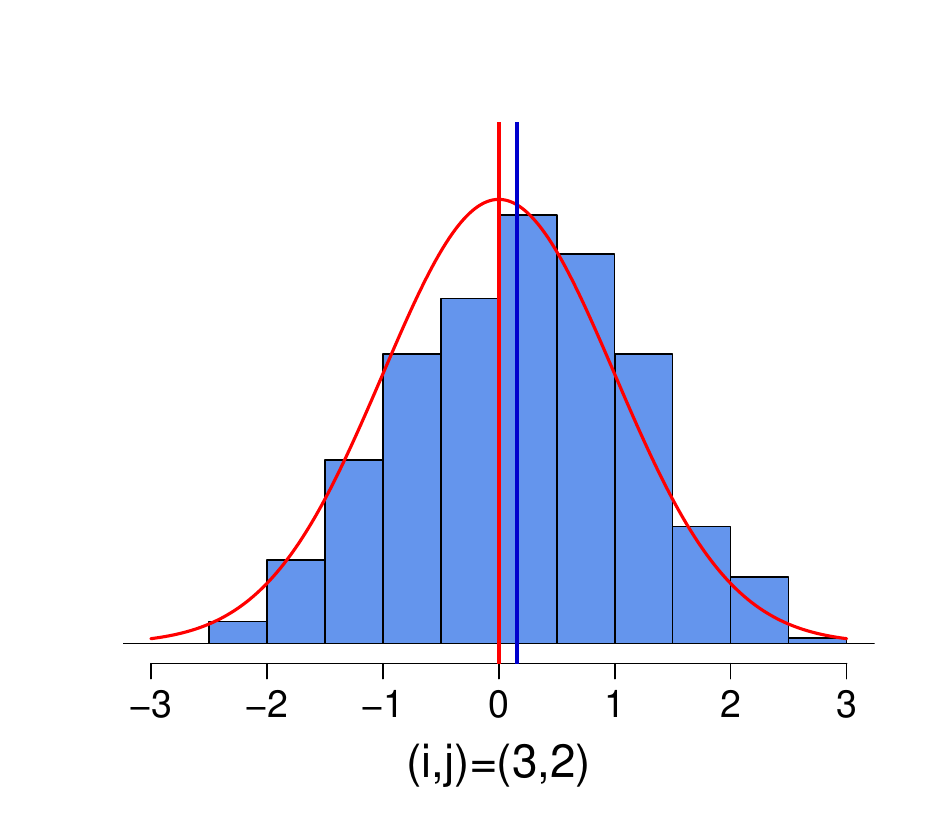}
    \end{minipage}
    \begin{minipage}{0.24\linewidth}
        \centering
        \includegraphics[width=\textwidth]{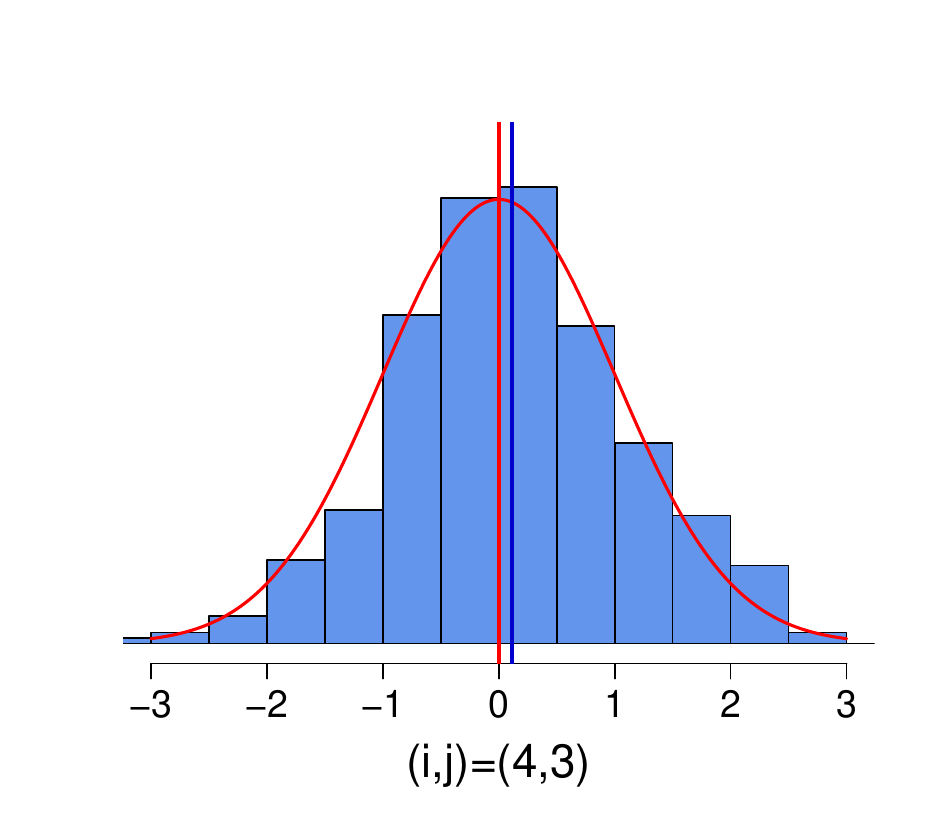}
    \end{minipage}
 \end{minipage}   
      \hspace{1cm}
 \begin{minipage}{0.3\linewidth}
    \begin{minipage}{0.24\linewidth}
        \centering
        \includegraphics[width=\textwidth]{images/L1_Gaussian_Band_n800_p200_12.pdf}
    \end{minipage}
    \begin{minipage}{0.24\linewidth}
        \centering
        \includegraphics[width=\textwidth]{images/L1_Gaussian_Band_n800_p200_23.pdf}
    \end{minipage}
    \begin{minipage}{0.24\linewidth}
        \centering
        \includegraphics[width=\textwidth]{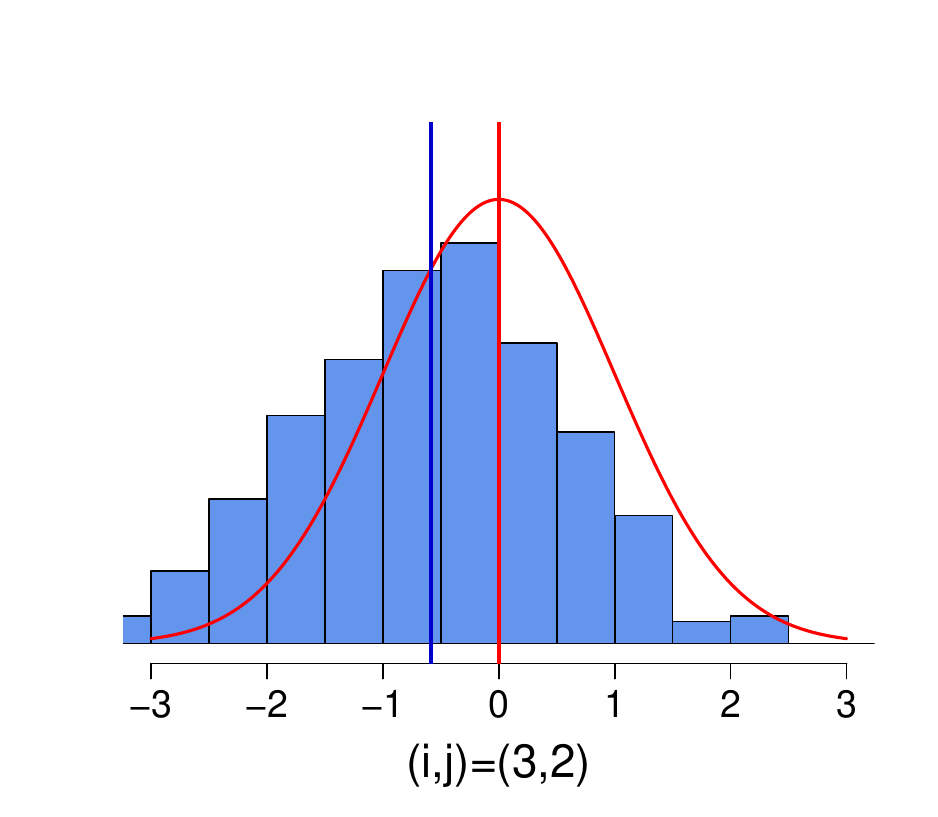}
    \end{minipage}
    \begin{minipage}{0.24\linewidth}
        \centering
        \includegraphics[width=\textwidth]{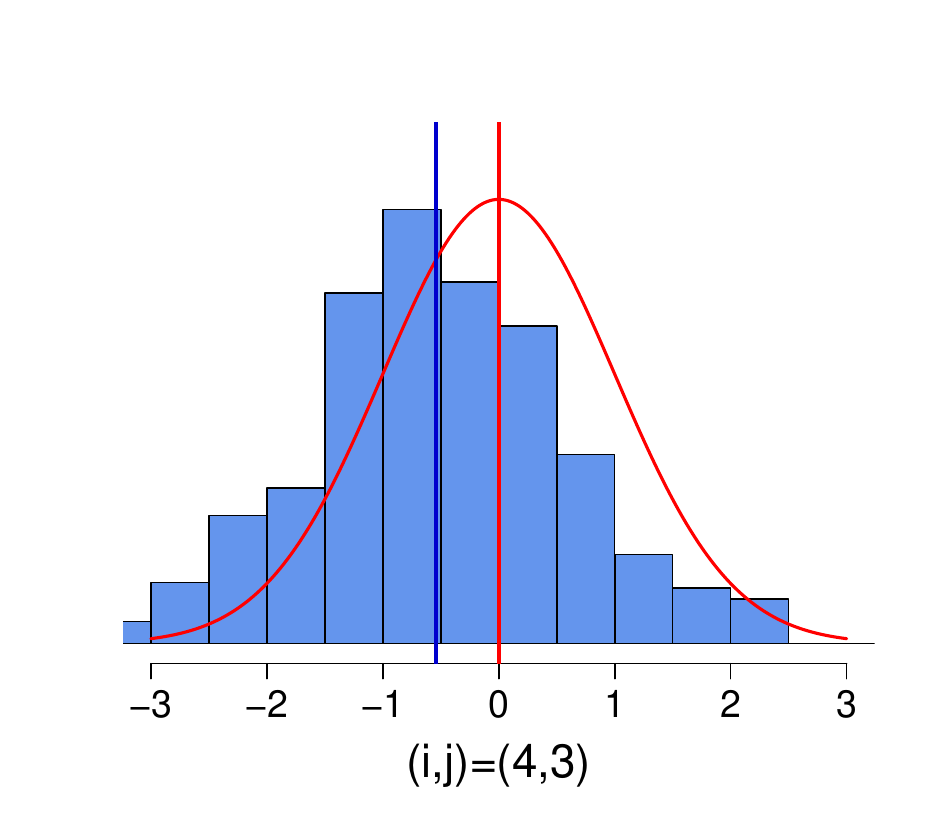}
    \end{minipage}
     \end{minipage}   
     
 \caption*{$n=200, p=400$}
   \vspace{-0.43cm}
 \begin{minipage}{0.3\linewidth}
    \begin{minipage}{0.24\linewidth}
        \centering
        \includegraphics[width=\textwidth]{images/unbias_L0_Gaussian_Band_n200_p400_12.pdf}
    \end{minipage}
    \begin{minipage}{0.24\linewidth}
        \centering
        \includegraphics[width=\textwidth]{images/unbias_L0_Gaussian_Band_n200_p400_23.pdf}
    \end{minipage}
    \begin{minipage}{0.24\linewidth}
        \centering
        \includegraphics[width=\textwidth]{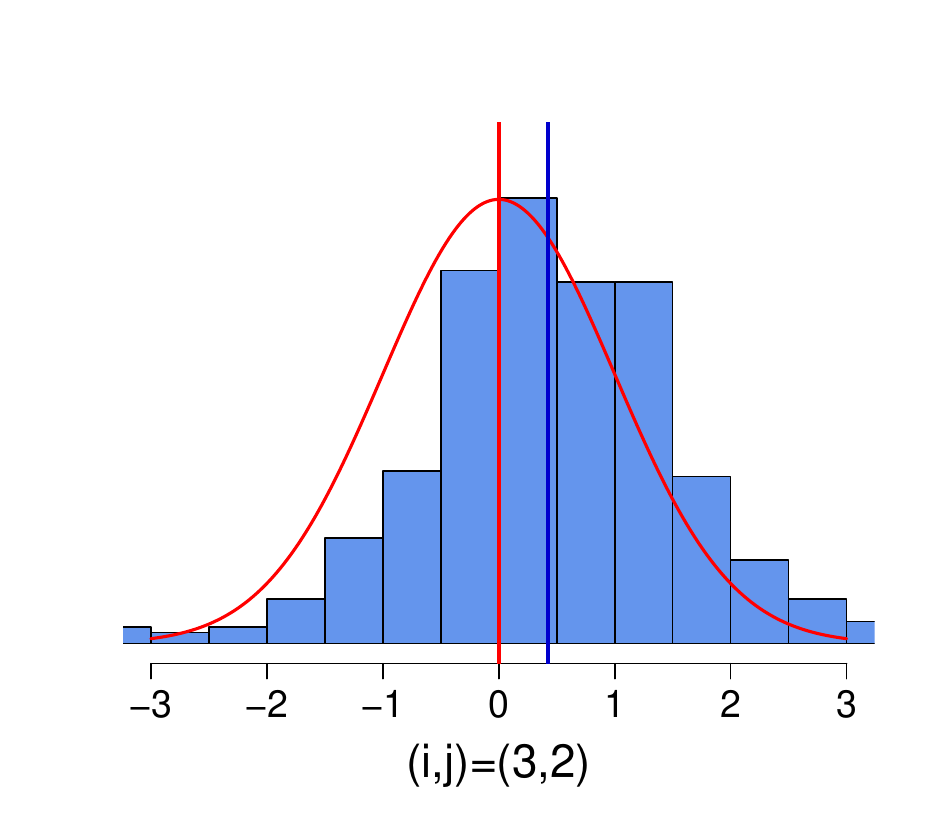}
    \end{minipage}
    \begin{minipage}{0.24\linewidth}
        \centering
        \includegraphics[width=\textwidth]{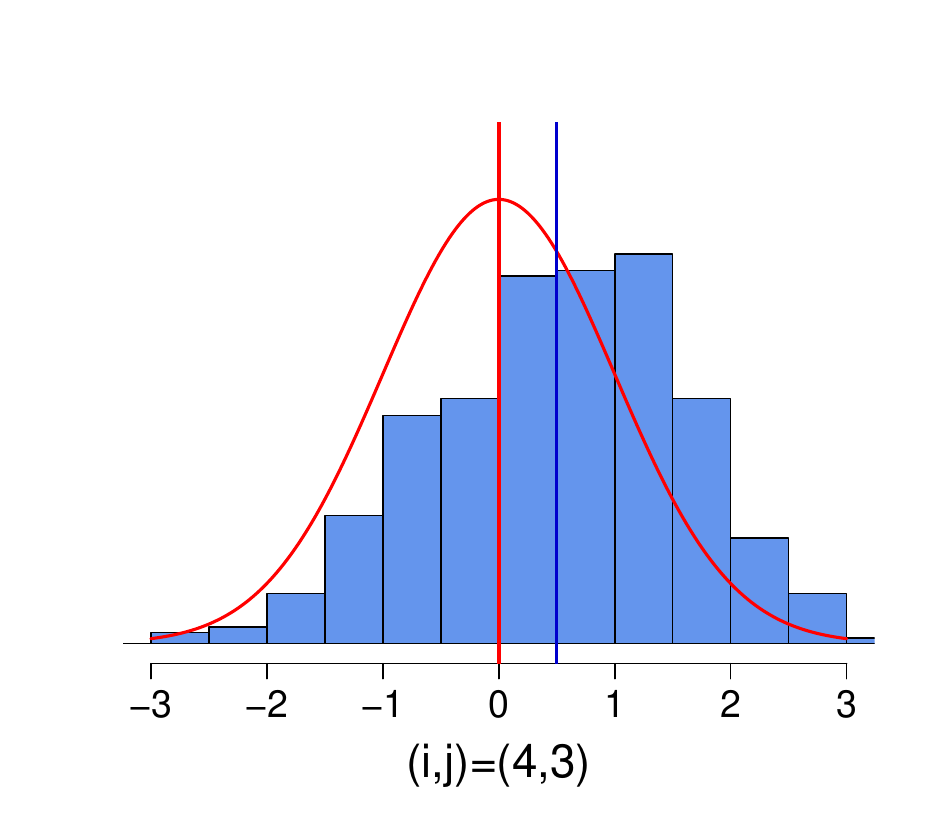}
    \end{minipage}
 \end{minipage}
 \hspace{1cm}
 \begin{minipage}{0.3\linewidth}
    \begin{minipage}{0.24\linewidth}
        \centering
        \includegraphics[width=\textwidth]{images/debias_L0_Gaussian_Band_n200_p400_12.pdf}
    \end{minipage}
    \begin{minipage}{0.24\linewidth}
        \centering
        \includegraphics[width=\textwidth]{images/debias_L0_Gaussian_Band_n200_p400_23.pdf}
    \end{minipage}
    \begin{minipage}{0.24\linewidth}
        \centering
        \includegraphics[width=\textwidth]{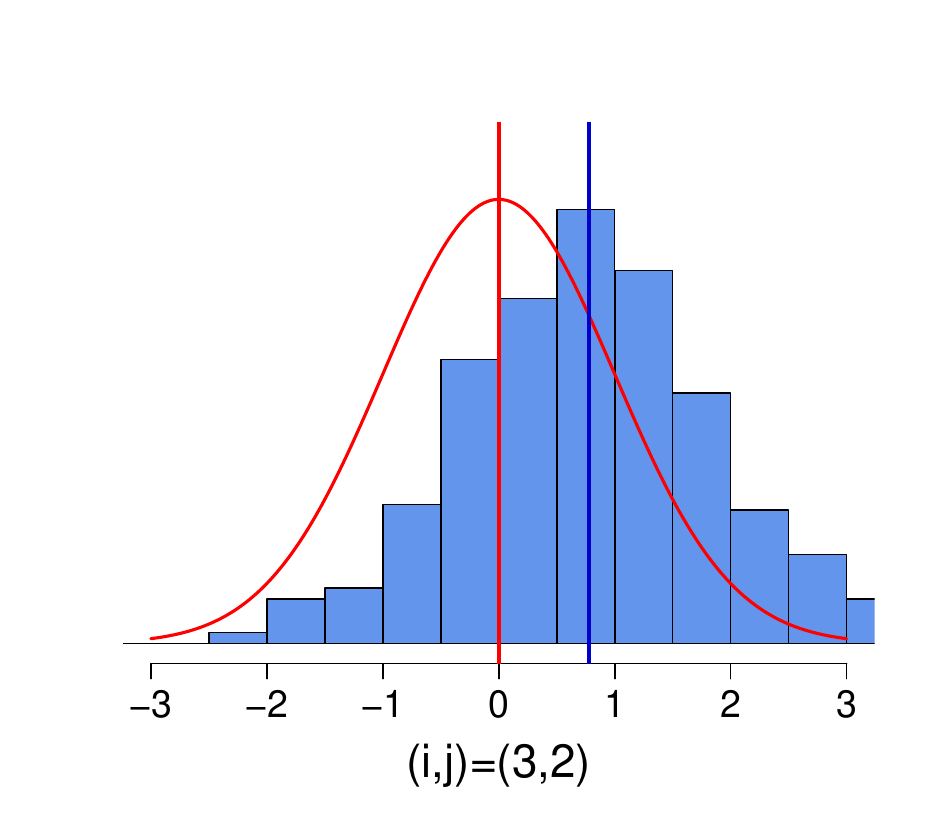}
    \end{minipage}
    \begin{minipage}{0.24\linewidth}
        \centering
        \includegraphics[width=\textwidth]{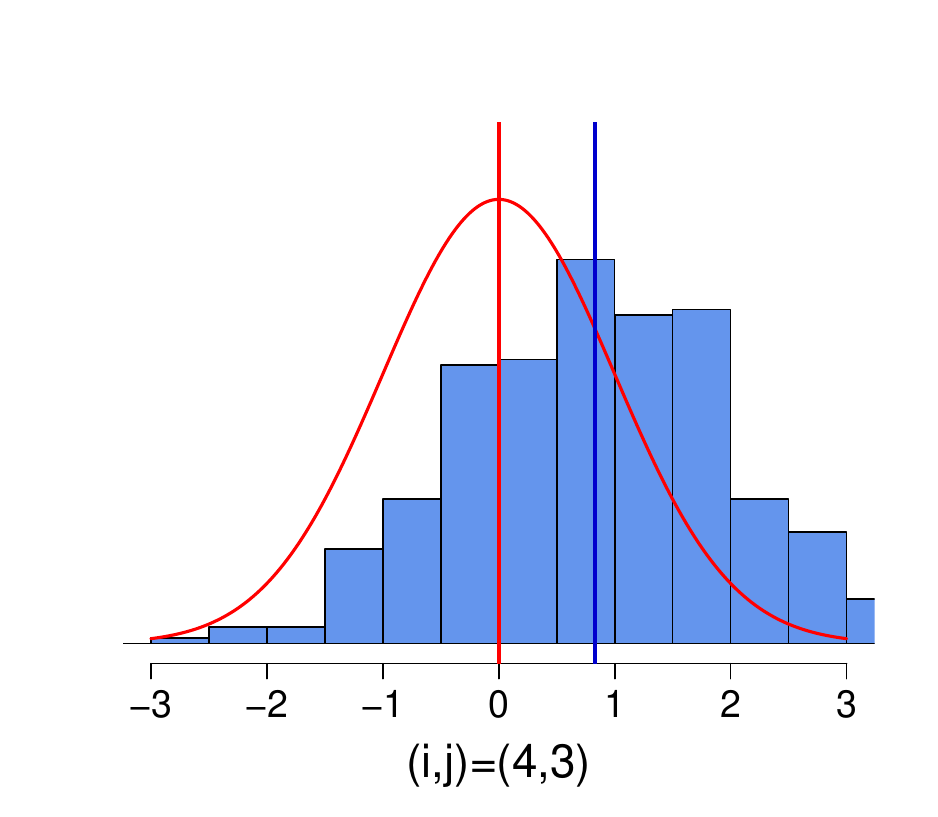}
    \end{minipage}    
 \end{minipage}
  \hspace{1cm}
 \begin{minipage}{0.3\linewidth}
     \begin{minipage}{0.24\linewidth}
        \centering
        \includegraphics[width=\textwidth]{images/L1_Gaussian_Band_n200_p400_12.pdf}
    \end{minipage}
    \begin{minipage}{0.24\linewidth}
        \centering
        \includegraphics[width=\textwidth]{images/L1_Gaussian_Band_n200_p400_23.pdf}
    \end{minipage}
    \begin{minipage}{0.24\linewidth}
        \centering
        \includegraphics[width=\textwidth]{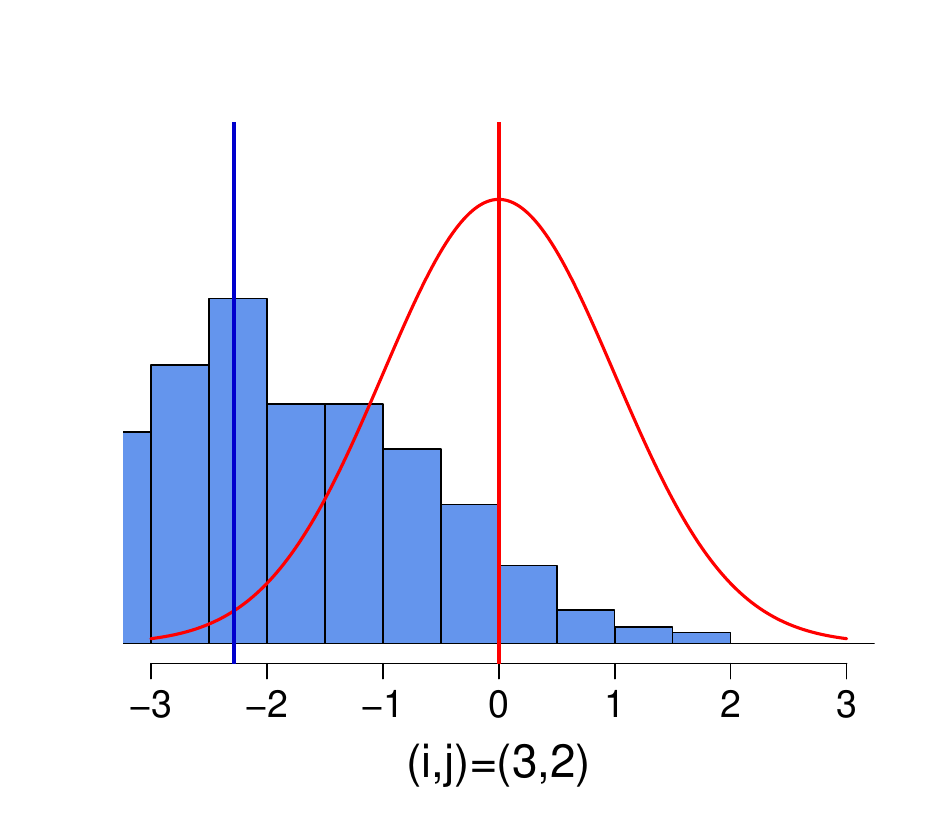}
    \end{minipage}
    \begin{minipage}{0.24\linewidth}
        \centering
        \includegraphics[width=\textwidth]{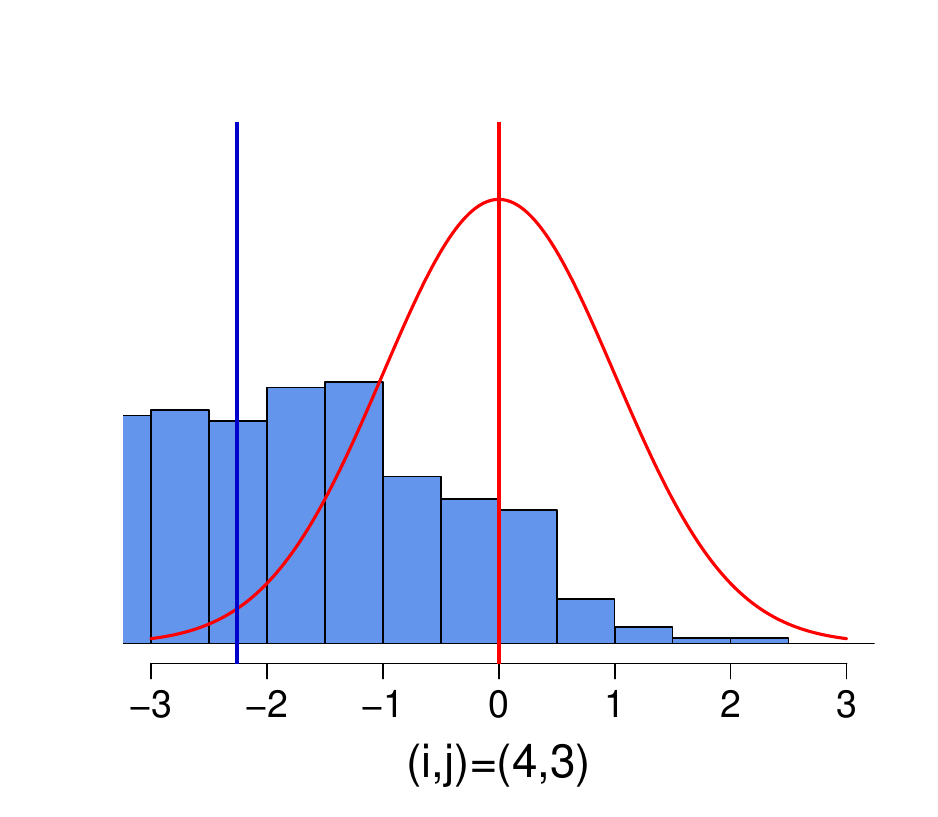}
    \end{minipage}
 \end{minipage}

  \caption*{$n=400, p=400$}
    \vspace{-0.43cm}
 \begin{minipage}{0.3\linewidth}
    \begin{minipage}{0.24\linewidth}
        \centering
        \includegraphics[width=\textwidth]{images/unbias_L0_Gaussian_Band_n400_p400_12.pdf}
    \end{minipage}
    \begin{minipage}{0.24\linewidth}
        \centering
        \includegraphics[width=\textwidth]{images/unbias_L0_Gaussian_Band_n400_p400_23.pdf}
    \end{minipage}
    \begin{minipage}{0.24\linewidth}
        \centering
        \includegraphics[width=\textwidth]{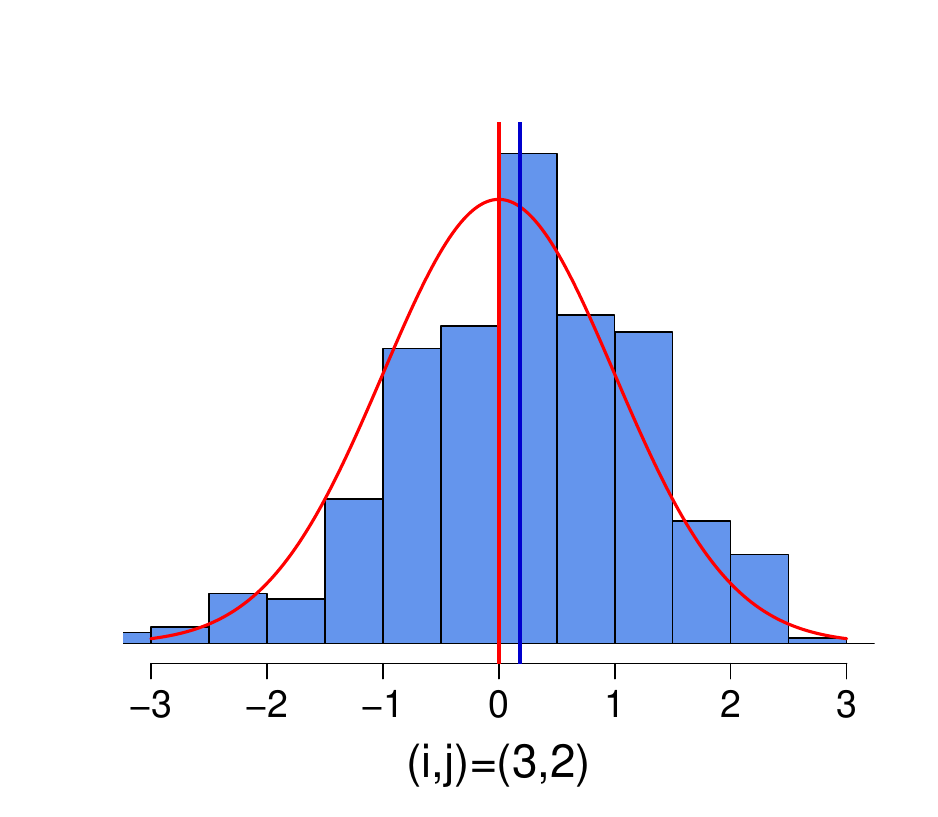}
    \end{minipage}
    \begin{minipage}{0.24\linewidth}
        \centering
        \includegraphics[width=\textwidth]{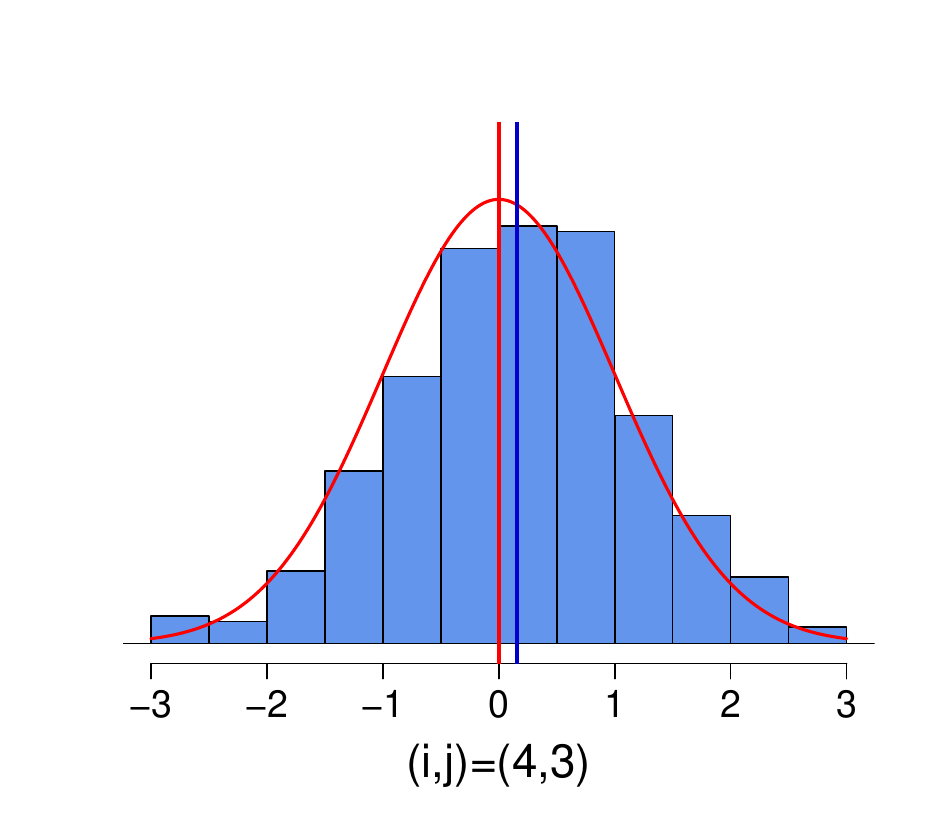}
    \end{minipage}
 \end{minipage}  
     \hspace{1cm}
 \begin{minipage}{0.3\linewidth}
    \begin{minipage}{0.24\linewidth}
        \centering
        \includegraphics[width=\textwidth]{images/debias_L0_Gaussian_Band_n400_p400_12.pdf}
    \end{minipage}
    \begin{minipage}{0.24\linewidth}
        \centering
        \includegraphics[width=\textwidth]{images/debias_L0_Gaussian_Band_n400_p400_23.pdf}
    \end{minipage}
    \begin{minipage}{0.24\linewidth}
        \centering
        \includegraphics[width=\textwidth]{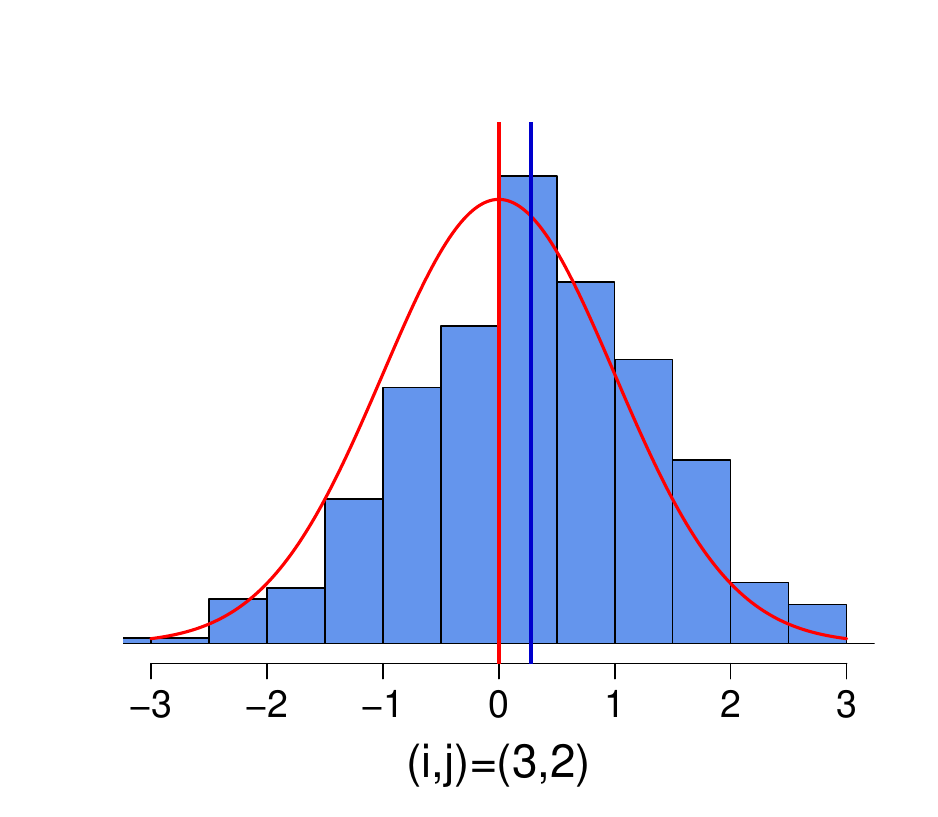}
    \end{minipage}
    \begin{minipage}{0.24\linewidth}
        \centering
        \includegraphics[width=\textwidth]{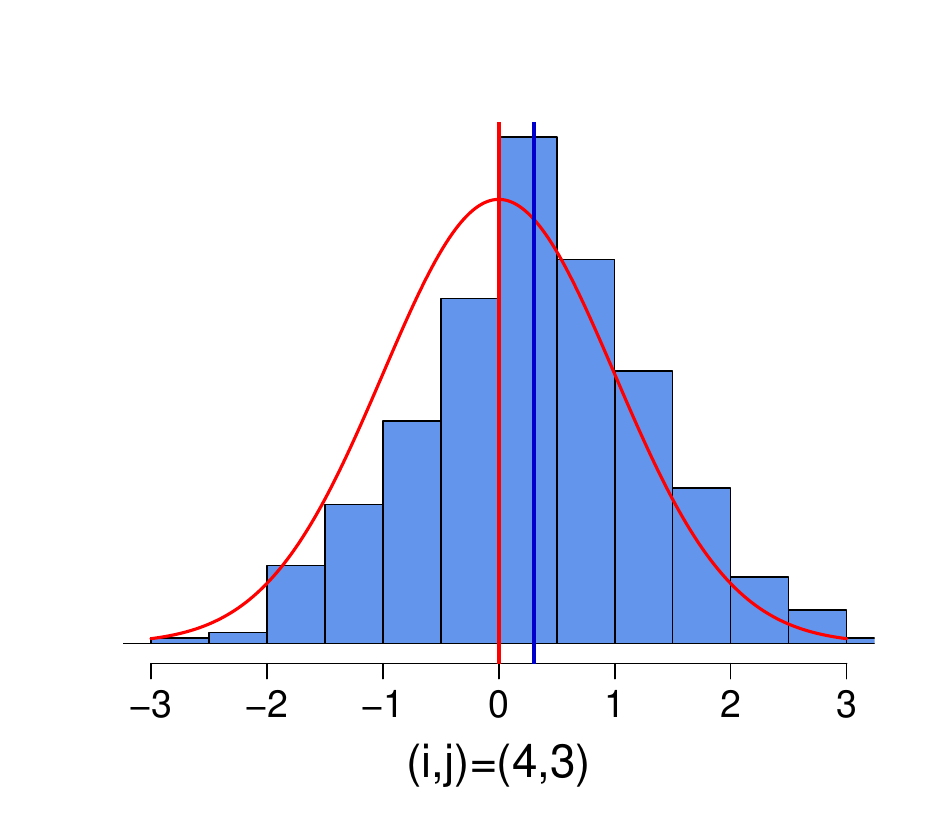}
    \end{minipage}
  \end{minipage}  
    \hspace{1cm}
 \begin{minipage}{0.3\linewidth}
    \begin{minipage}{0.24\linewidth}
        \centering
        \includegraphics[width=\textwidth]{images/L1_Gaussian_Band_n400_p400_12.pdf}
    \end{minipage}
    \begin{minipage}{0.24\linewidth}
        \centering
        \includegraphics[width=\textwidth]{images/L1_Gaussian_Band_n400_p400_23.pdf}
    \end{minipage}
    \begin{minipage}{0.24\linewidth}
        \centering
        \includegraphics[width=\textwidth]{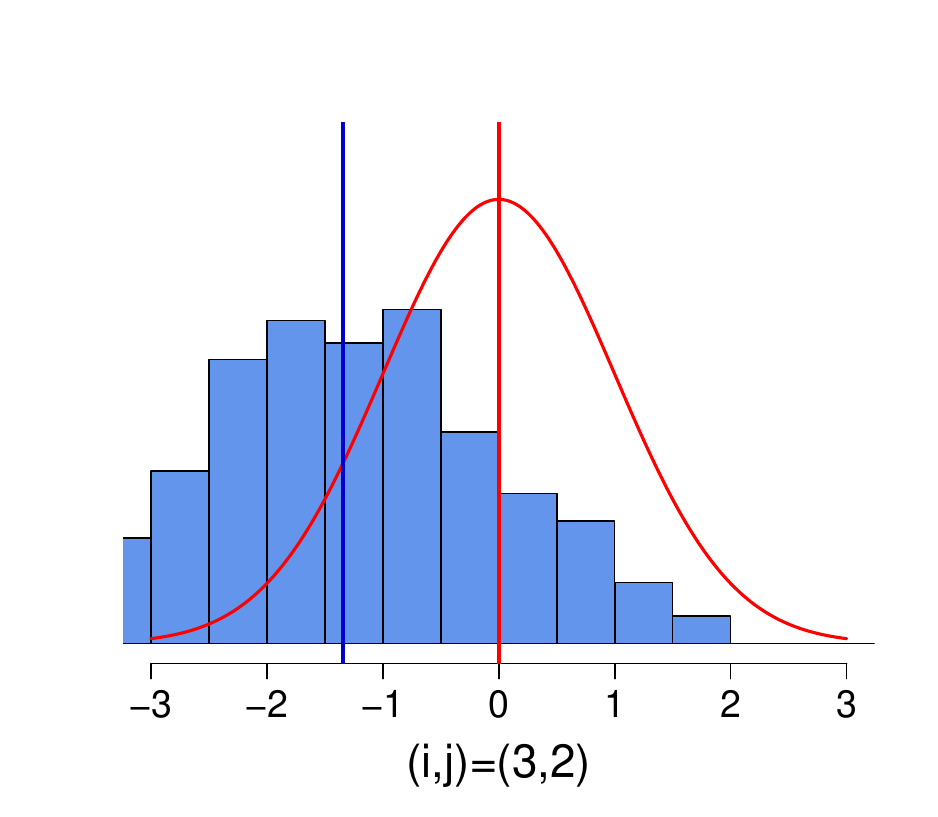}
    \end{minipage}
    \begin{minipage}{0.24\linewidth}
        \centering
        \includegraphics[width=\textwidth]{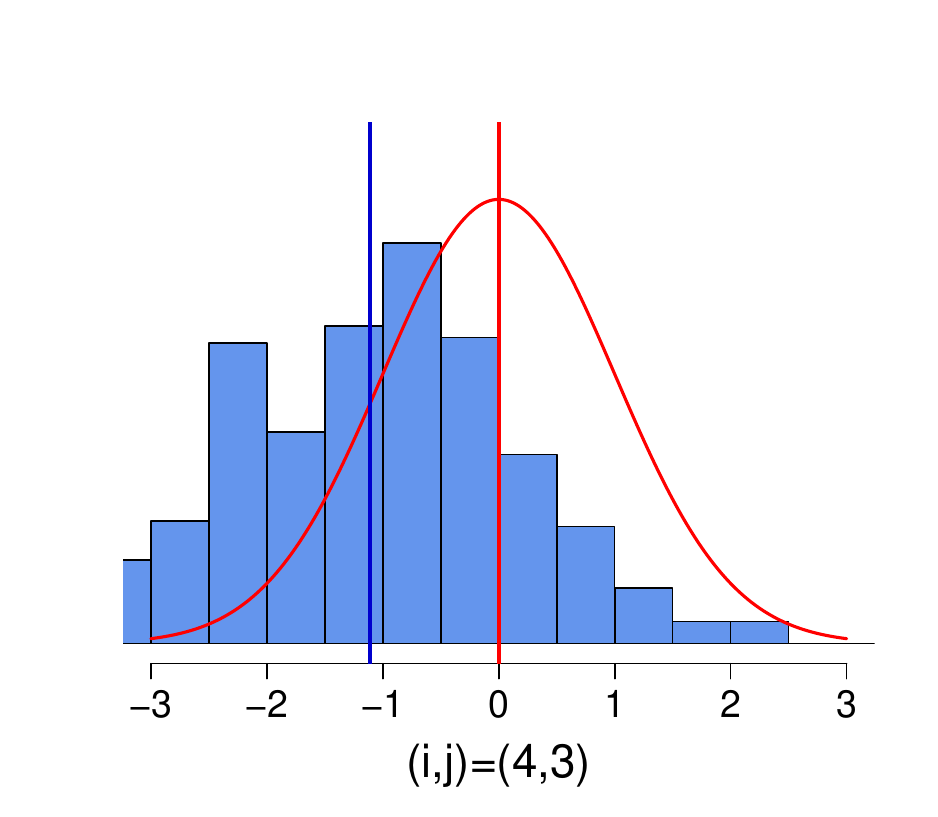}
    \end{minipage}
 \end{minipage}

  \caption*{$n=800, p=400$}
    \vspace{-0.43cm}
 \begin{minipage}{0.3\linewidth}
    \begin{minipage}{0.24\linewidth}
        \centering
        \includegraphics[width=\textwidth]{images/unbias_L0_Gaussian_Band_n800_p400_12.pdf}
    \end{minipage}
    \begin{minipage}{0.24\linewidth}
        \centering
        \includegraphics[width=\textwidth]{images/unbias_L0_Gaussian_Band_n800_p400_23.pdf}
    \end{minipage}
    \begin{minipage}{0.24\linewidth}
        \centering
        \includegraphics[width=\textwidth]{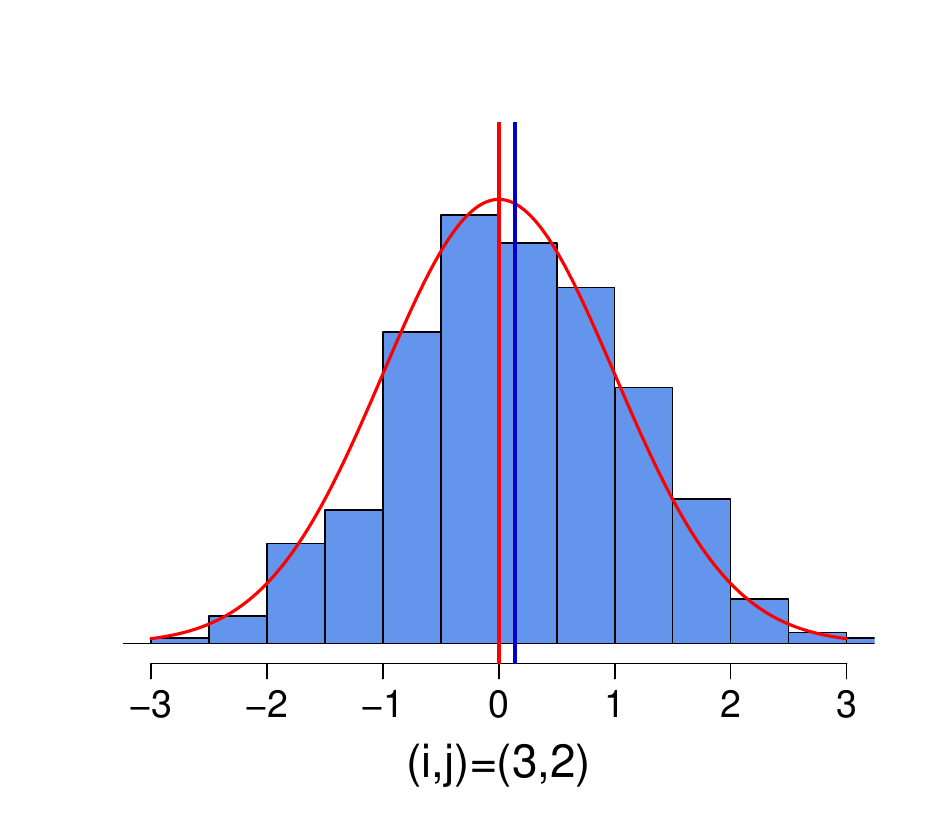}
    \end{minipage}
    \begin{minipage}{0.24\linewidth}
        \centering
        \includegraphics[width=\textwidth]{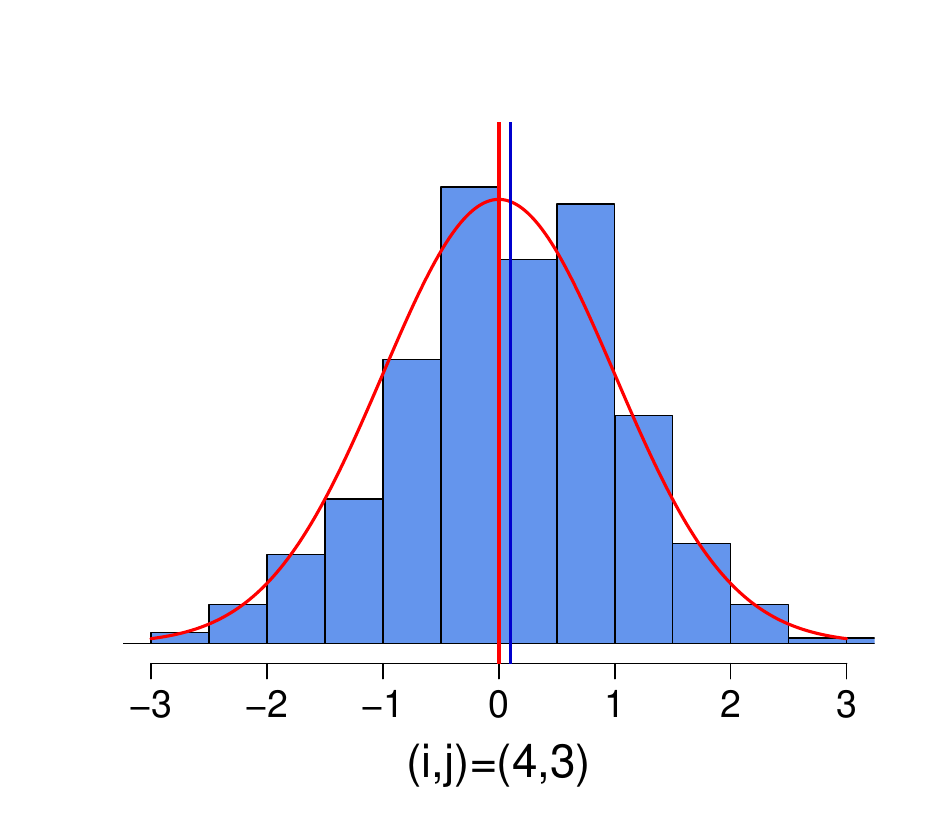}
    \end{minipage}
   \caption*{(a)~~$L_0{:}~ \widehat{\mb{\Omega}}^{\text{US}}$}
 \end{minipage} 
     \hspace{1cm}
 \begin{minipage}{0.3\linewidth}
    \begin{minipage}{0.24\linewidth}
        \centering
        \includegraphics[width=\textwidth]{images/debias_L0_Gaussian_Band_n800_p400_12.pdf}
    \end{minipage}
    \begin{minipage}{0.24\linewidth}
        \centering
        \includegraphics[width=\textwidth]{images/debias_L0_Gaussian_Band_n800_p400_23.pdf}
    \end{minipage}
    \begin{minipage}{0.24\linewidth}
        \centering
        \includegraphics[width=\textwidth]{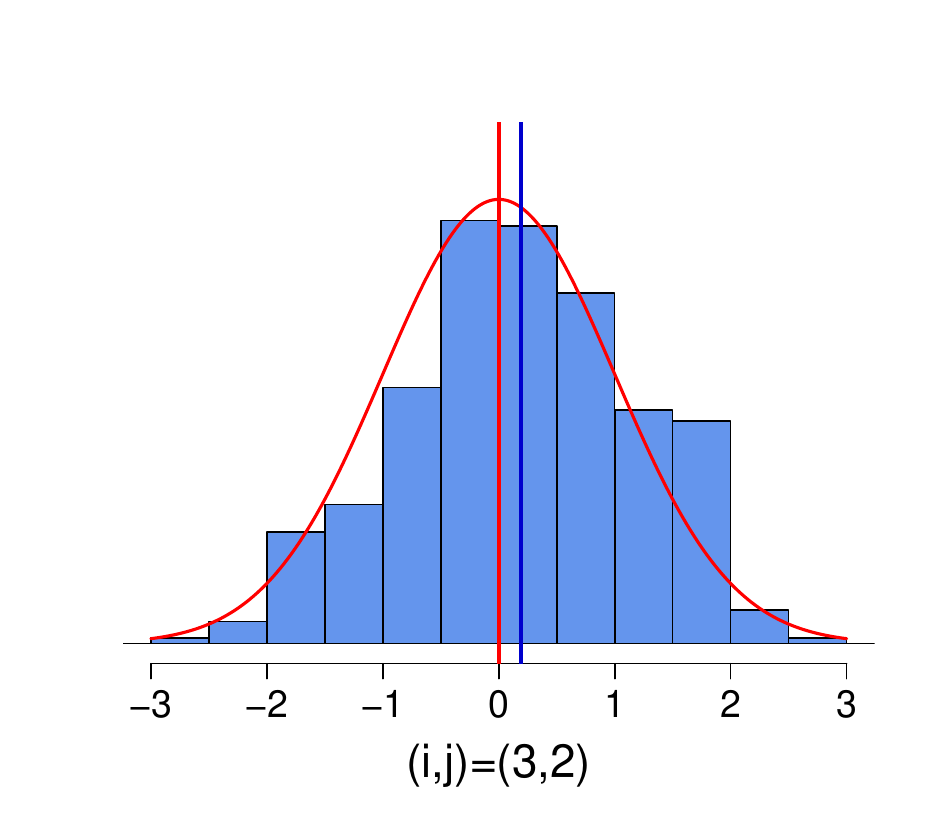}
    \end{minipage}
    \begin{minipage}{0.24\linewidth}
        \centering
        \includegraphics[width=\textwidth]{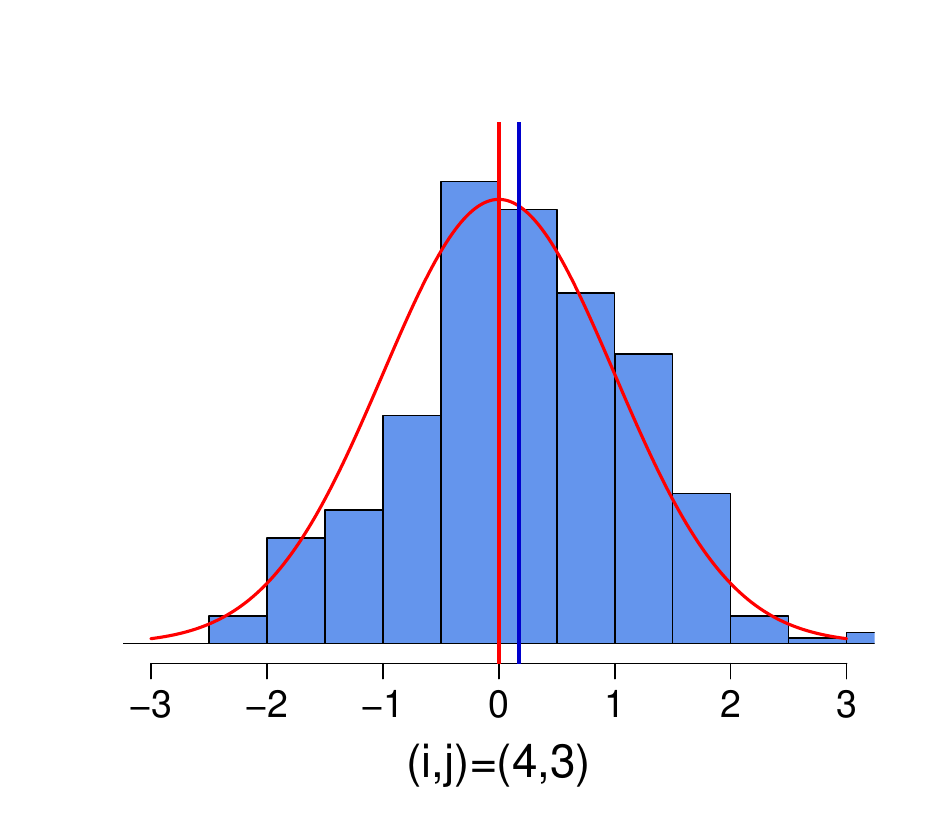}
    \end{minipage}
        \caption*{(b)~~$L_0{:}~ \widehat{\mb{T}}$}
 \end{minipage}   
      \hspace{1cm}
 \begin{minipage}{0.3\linewidth}
    \begin{minipage}{0.24\linewidth}
        \centering
        \includegraphics[width=\textwidth]{images/L1_Gaussian_Band_n800_p400_12.pdf}
    \end{minipage}
    \begin{minipage}{0.24\linewidth}
        \centering
        \includegraphics[width=\textwidth]{images/L1_Gaussian_Band_n800_p400_23.pdf}
    \end{minipage}
    \begin{minipage}{0.24\linewidth}
        \centering
        \includegraphics[width=\textwidth]{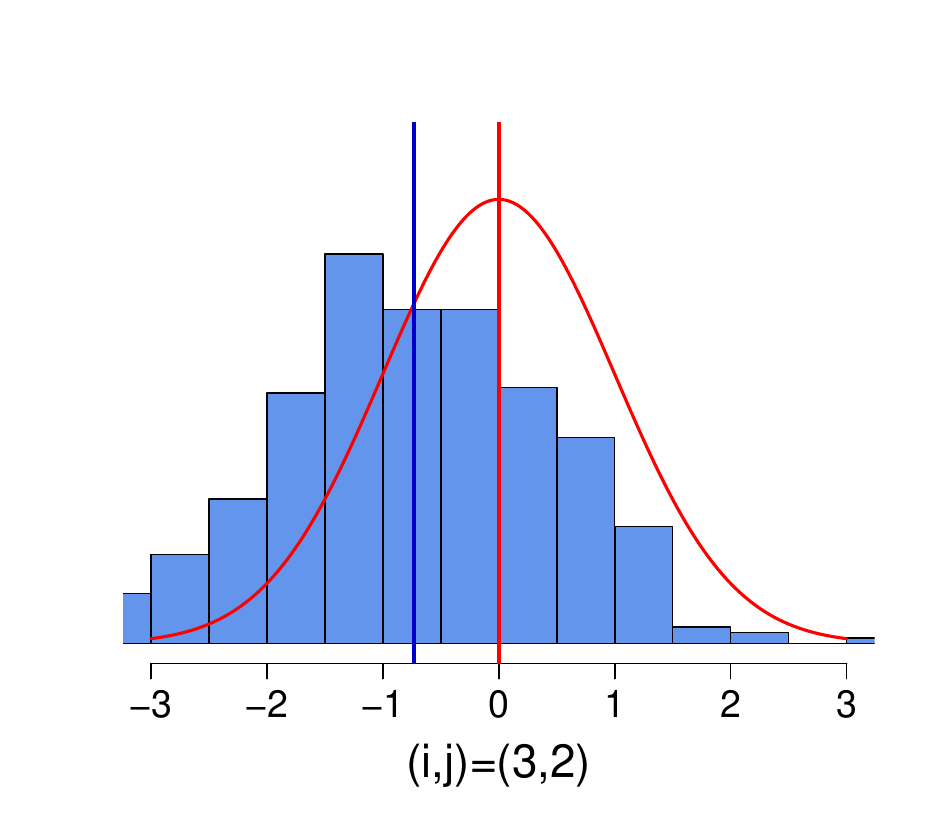}
    \end{minipage}
    \begin{minipage}{0.24\linewidth}
        \centering
        \includegraphics[width=\textwidth]{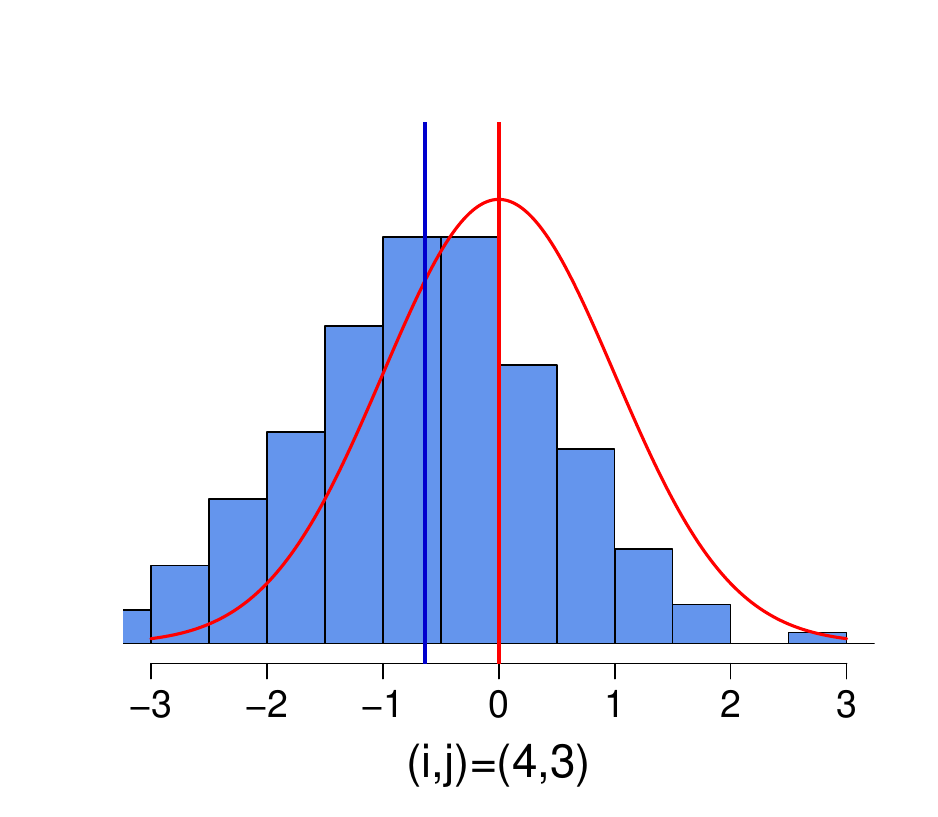}
    \end{minipage}
    \caption*{(c)~~$L_1{:}~ \widehat{\mb{T}}$}
     \end{minipage}   
     \caption{Histograms of $\big(\sqrt{n}(\widehat{\mb{\Omega}}_{ij}^{(m)}-\mb{\Omega}_{ij})/\widehat{\sigma}_{\mb{\Omega}_{ij}}^{(m)}\big)_{m=1}^{400}$ under Gaussian band graph settings.}
\label{fig: normalplot Gaussian band}
\end{sidewaysfigure}

%Gaussian random
 \begin{sidewaysfigure}[th!]
  \caption*{$n=200, p=200$}
      \vspace{-0.43cm}
 \begin{minipage}{0.3\linewidth}
    \begin{minipage}{0.24\linewidth}
        \centering
        \includegraphics[width=\textwidth]{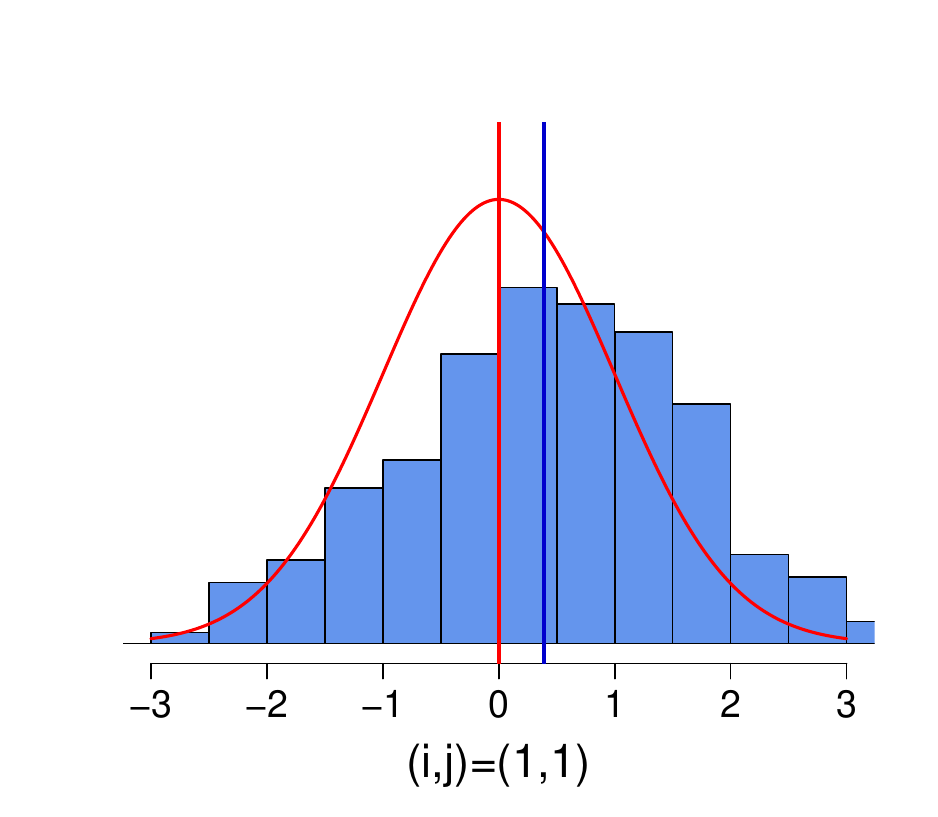}
    \end{minipage}
    \begin{minipage}{0.24\linewidth}
        \centering
        \includegraphics[width=\textwidth]{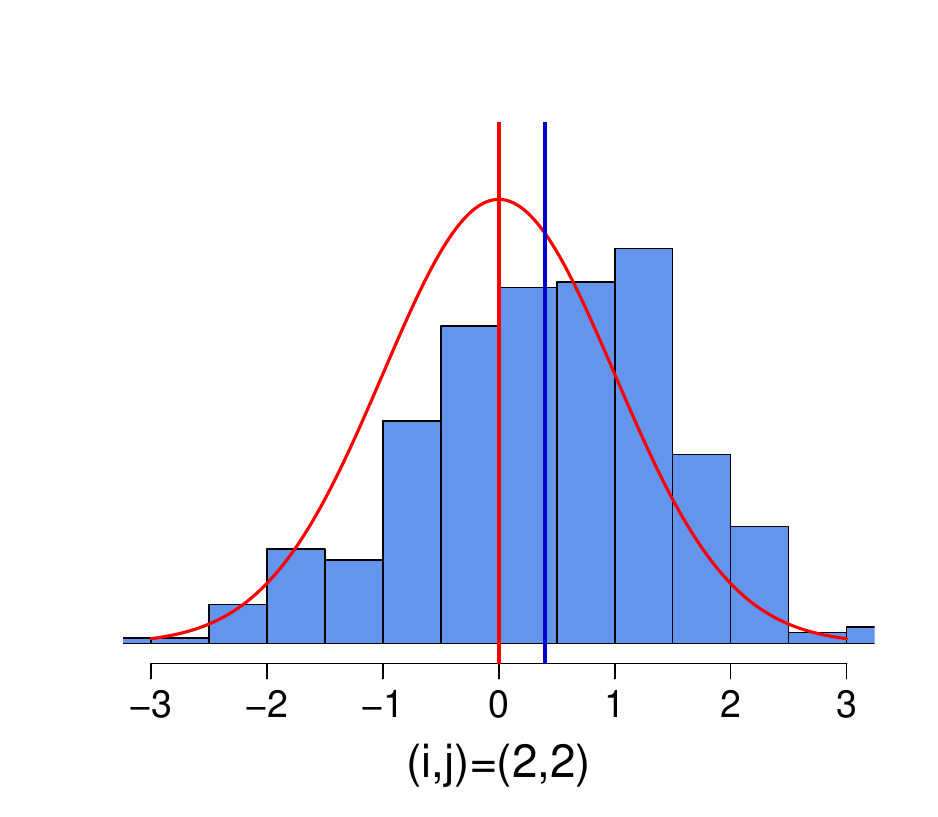}
    \end{minipage}
    \begin{minipage}{0.24\linewidth}
        \centering
        \includegraphics[width=\textwidth]{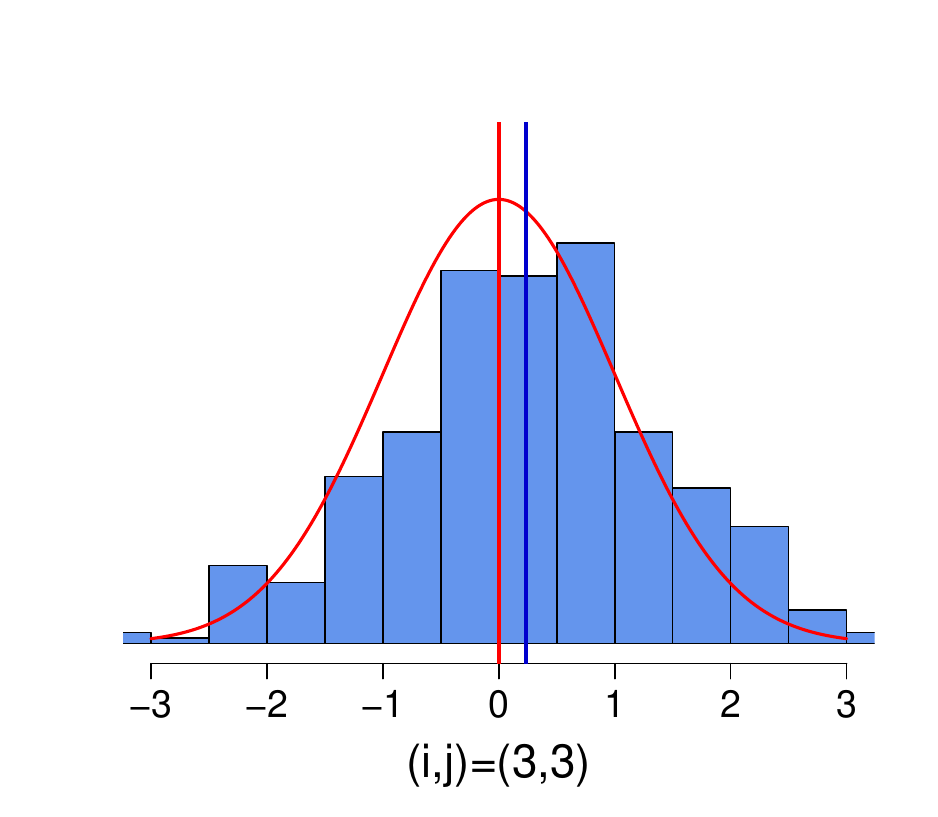}
    \end{minipage}
    \begin{minipage}{0.24\linewidth}
        \centering
        \includegraphics[width=\textwidth]{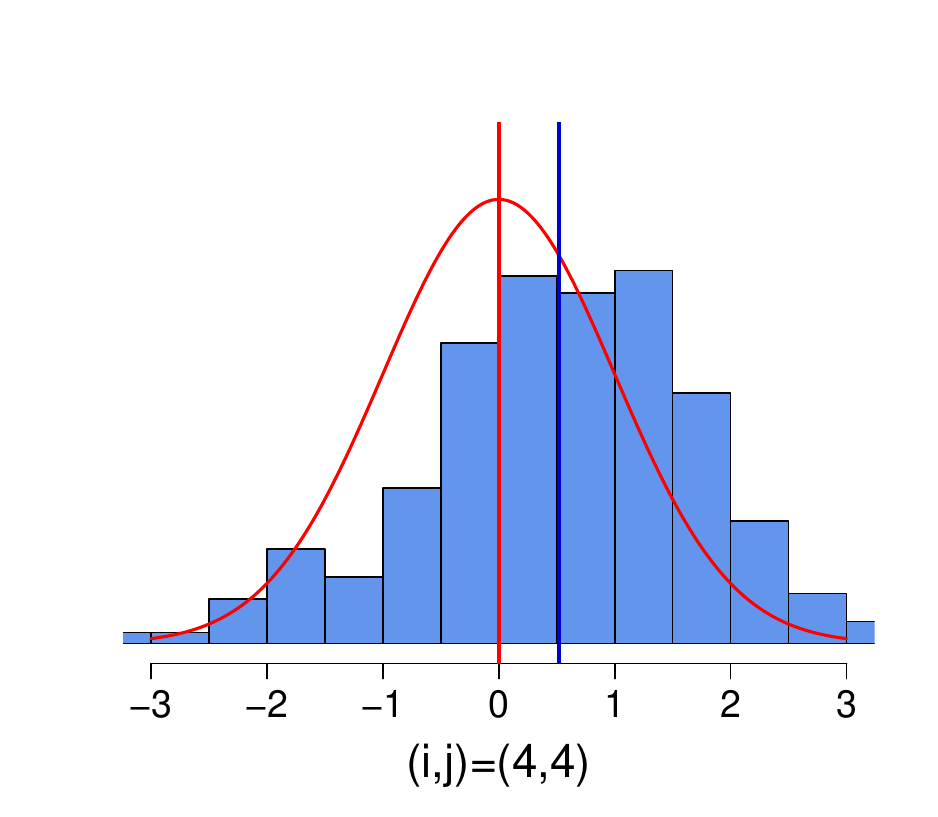}
    \end{minipage}
 \end{minipage}
 \hspace{1cm}
 \begin{minipage}{0.3\linewidth}
    \begin{minipage}{0.24\linewidth}
        \centering
        \includegraphics[width=\textwidth]{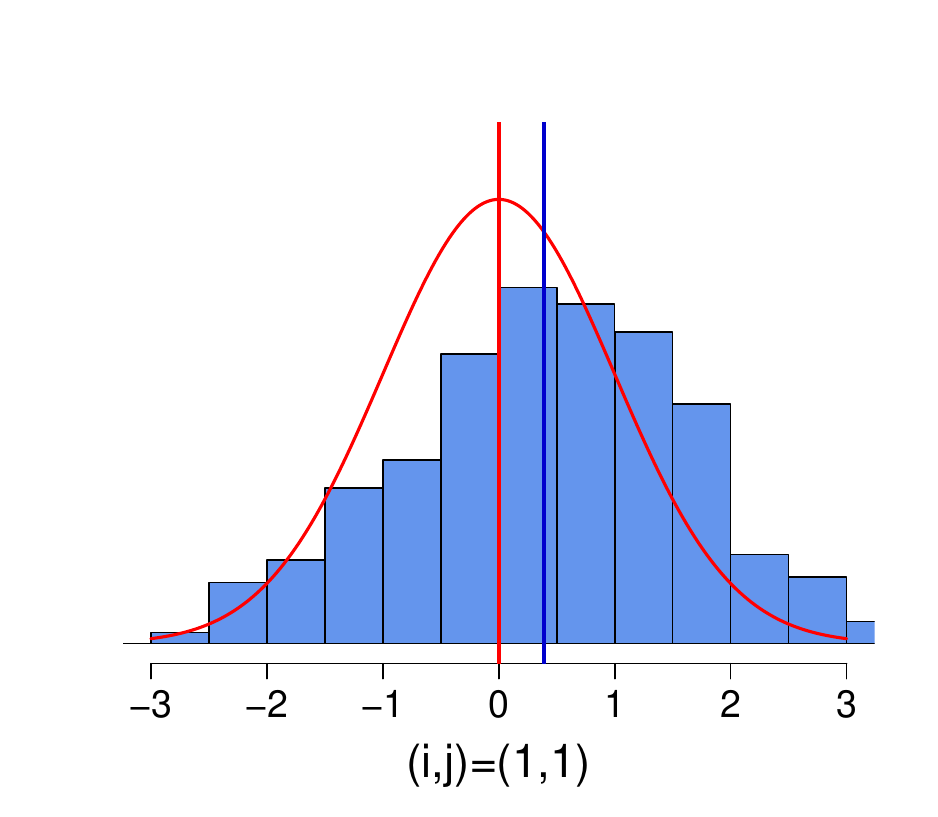}
    \end{minipage}
    \begin{minipage}{0.24\linewidth}
        \centering
        \includegraphics[width=\textwidth]{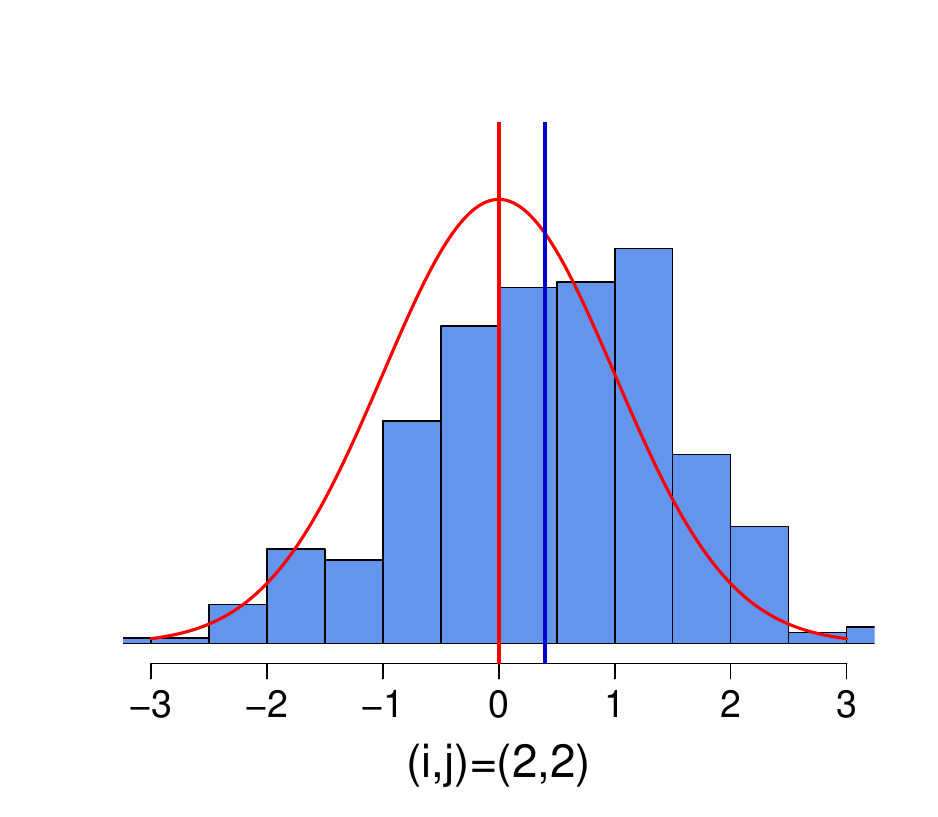}
    \end{minipage}
    \begin{minipage}{0.24\linewidth}
        \centering
        \includegraphics[width=\textwidth]{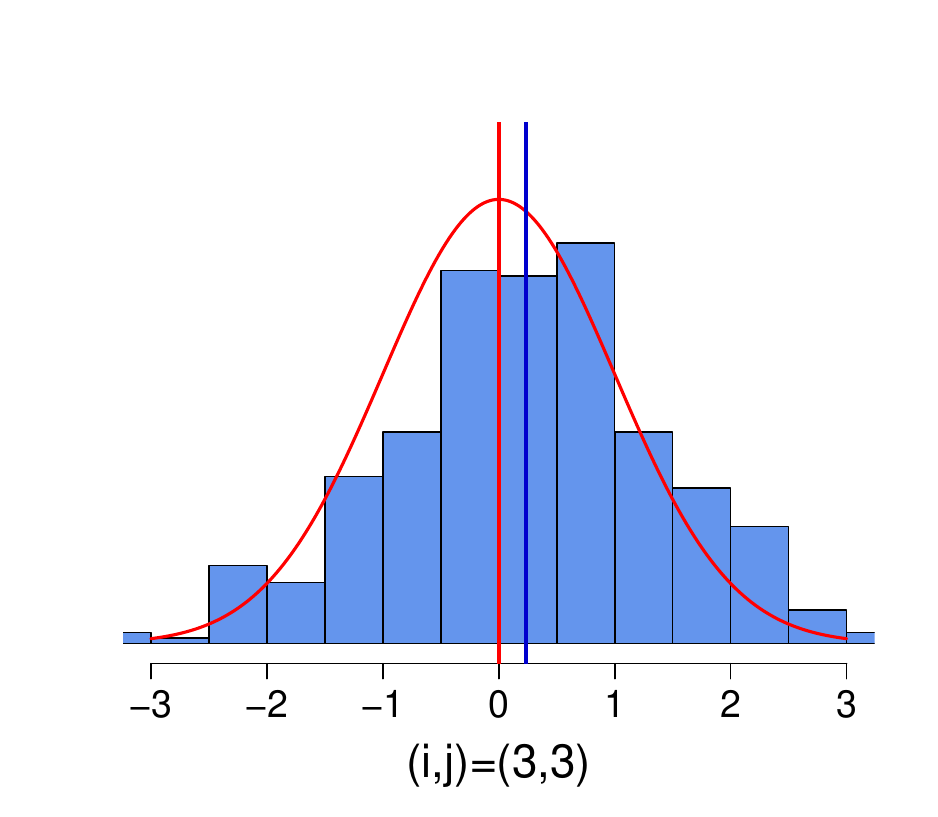}
    \end{minipage}
    \begin{minipage}{0.24\linewidth}
        \centering
        \includegraphics[width=\textwidth]{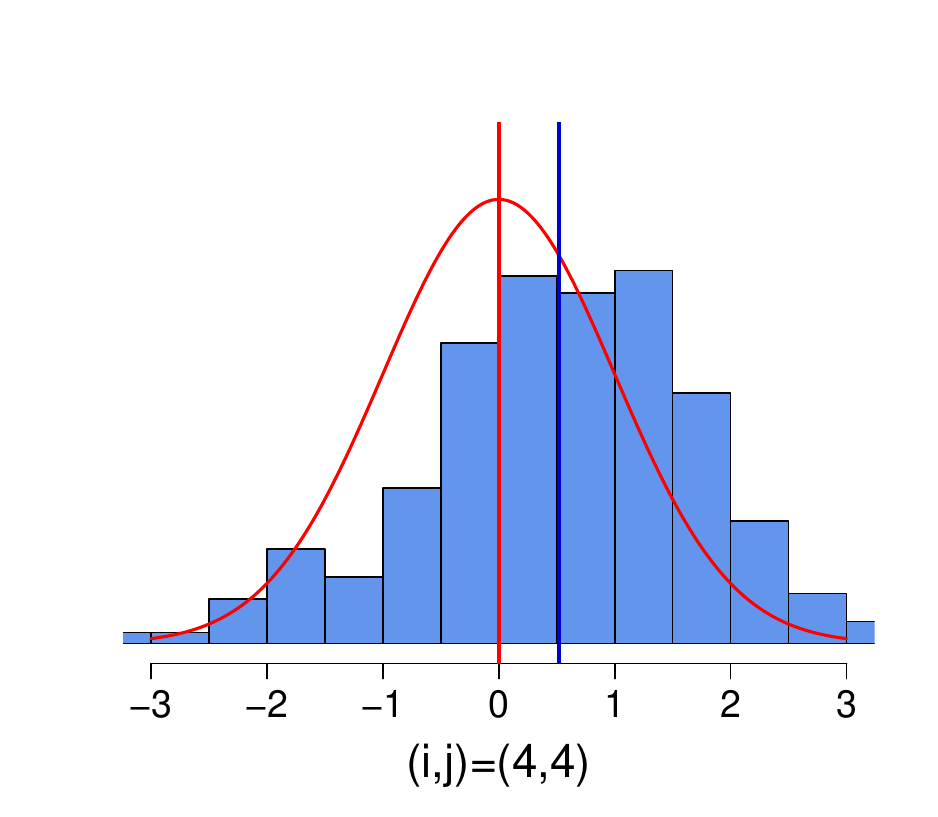}
    \end{minipage}    
 \end{minipage}
  \hspace{1cm}
 \begin{minipage}{0.3\linewidth}
     \begin{minipage}{0.24\linewidth}
        \centering
        \includegraphics[width=\textwidth]{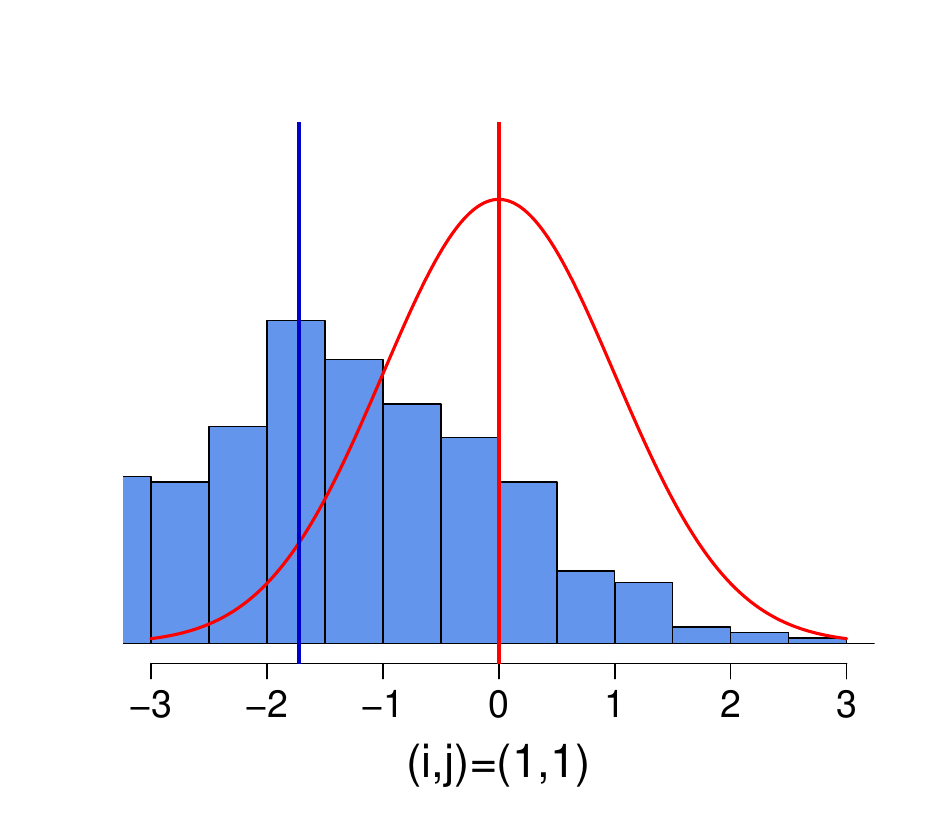}
    \end{minipage}
    \begin{minipage}{0.24\linewidth}
        \centering
        \includegraphics[width=\textwidth]{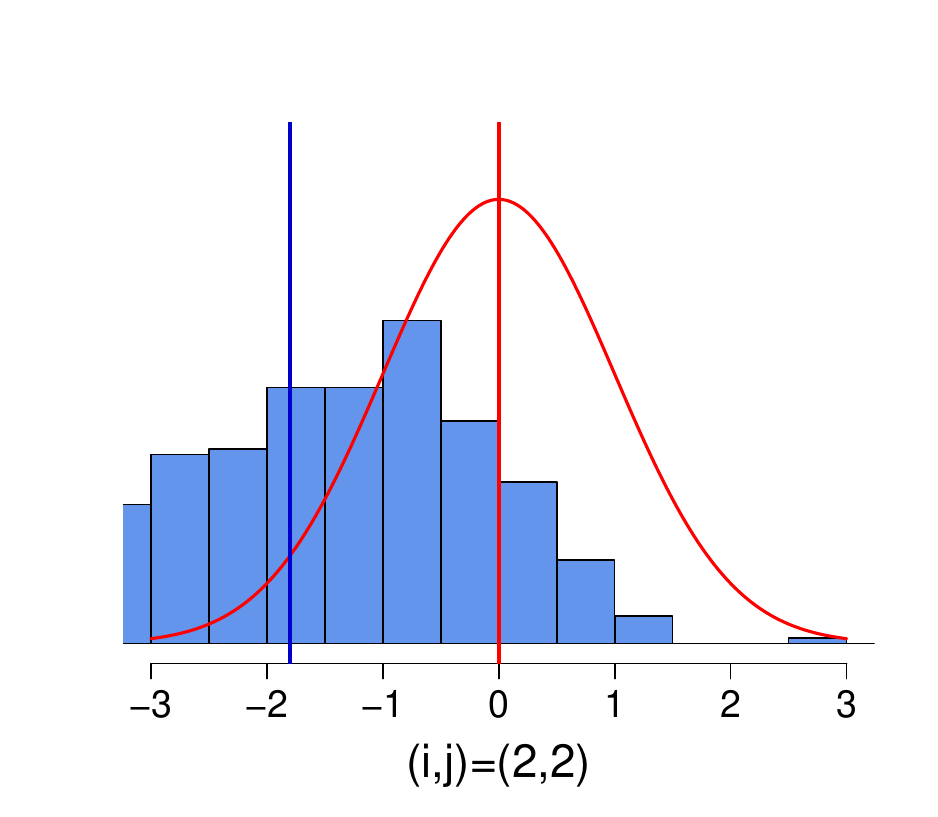}
    \end{minipage}
    \begin{minipage}{0.24\linewidth}
        \centering
        \includegraphics[width=\textwidth]{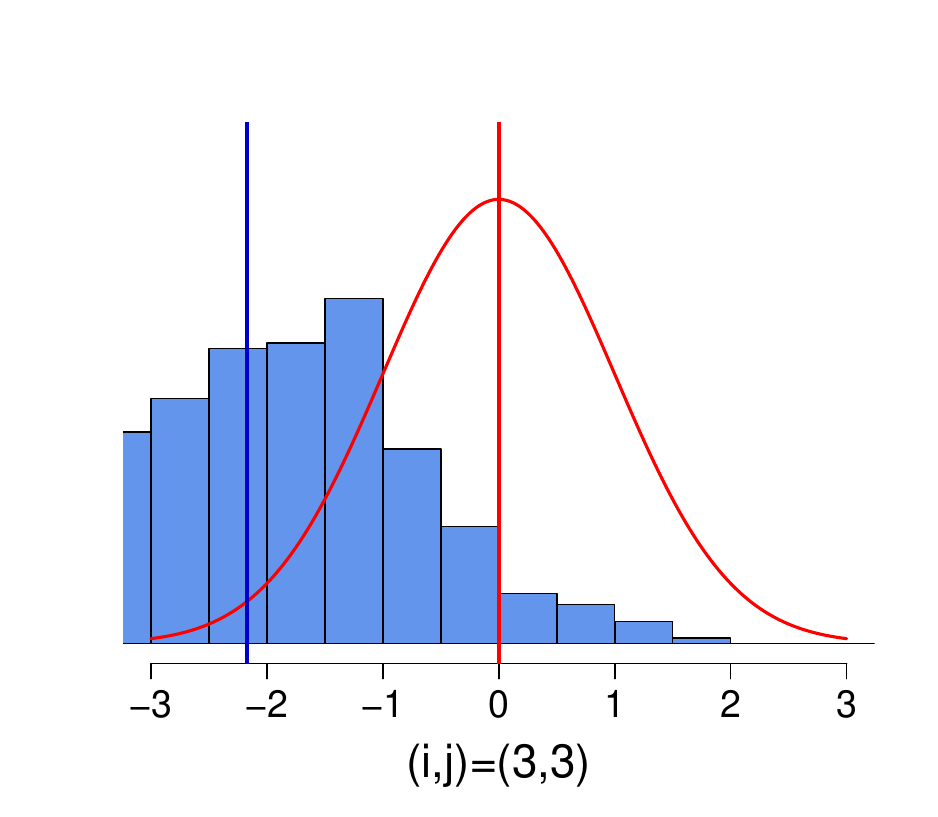}
    \end{minipage}
    \begin{minipage}{0.24\linewidth}
        \centering
        \includegraphics[width=\textwidth]{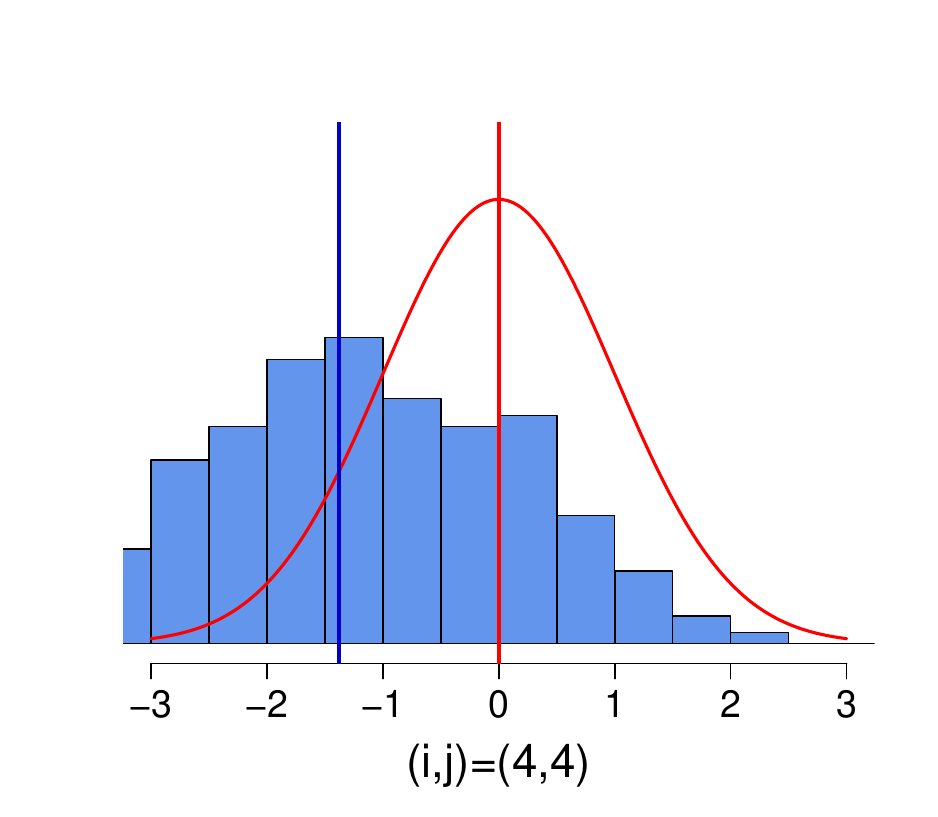}
    \end{minipage}
 \end{minipage}

  \caption*{$n=400, p=200$}
      \vspace{-0.43cm}
 \begin{minipage}{0.3\linewidth}
    \begin{minipage}{0.24\linewidth}
        \centering
        \includegraphics[width=\textwidth]{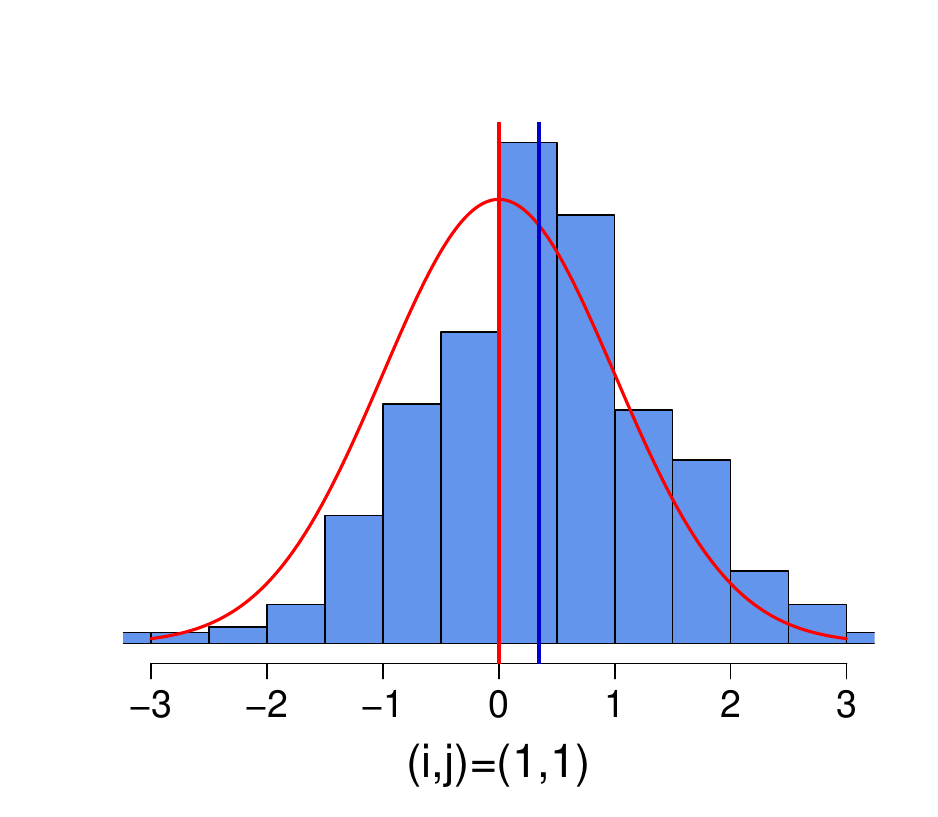}
    \end{minipage}
    \begin{minipage}{0.24\linewidth}
        \centering
        \includegraphics[width=\textwidth]{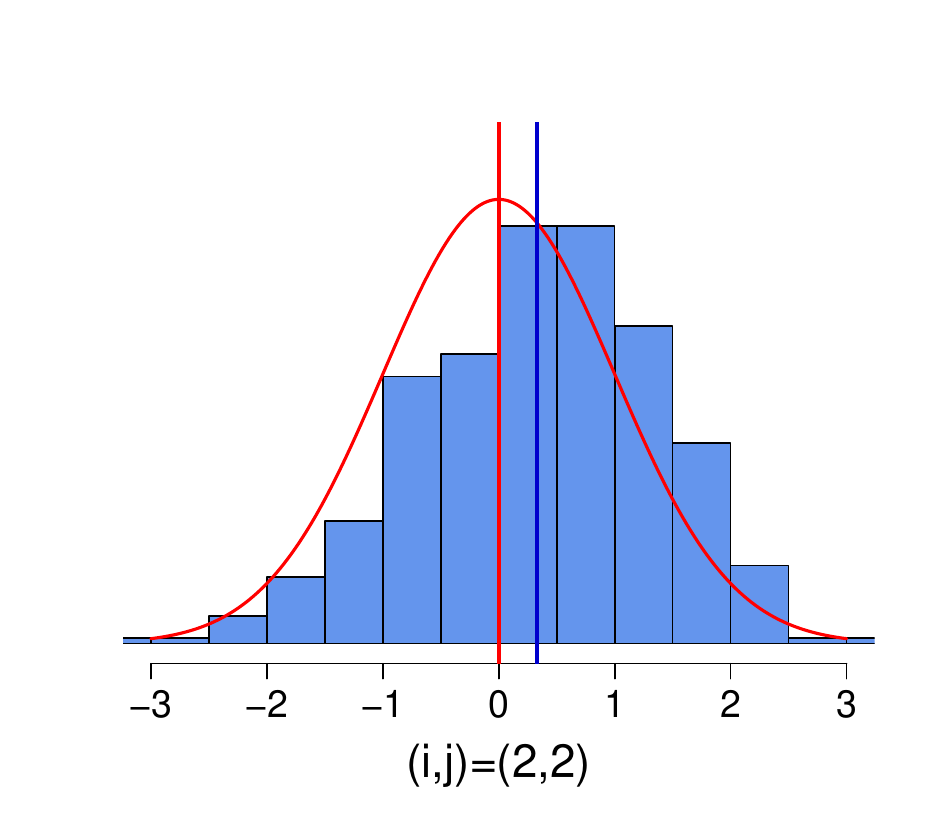}
    \end{minipage}
    \begin{minipage}{0.24\linewidth}
        \centering
        \includegraphics[width=\textwidth]{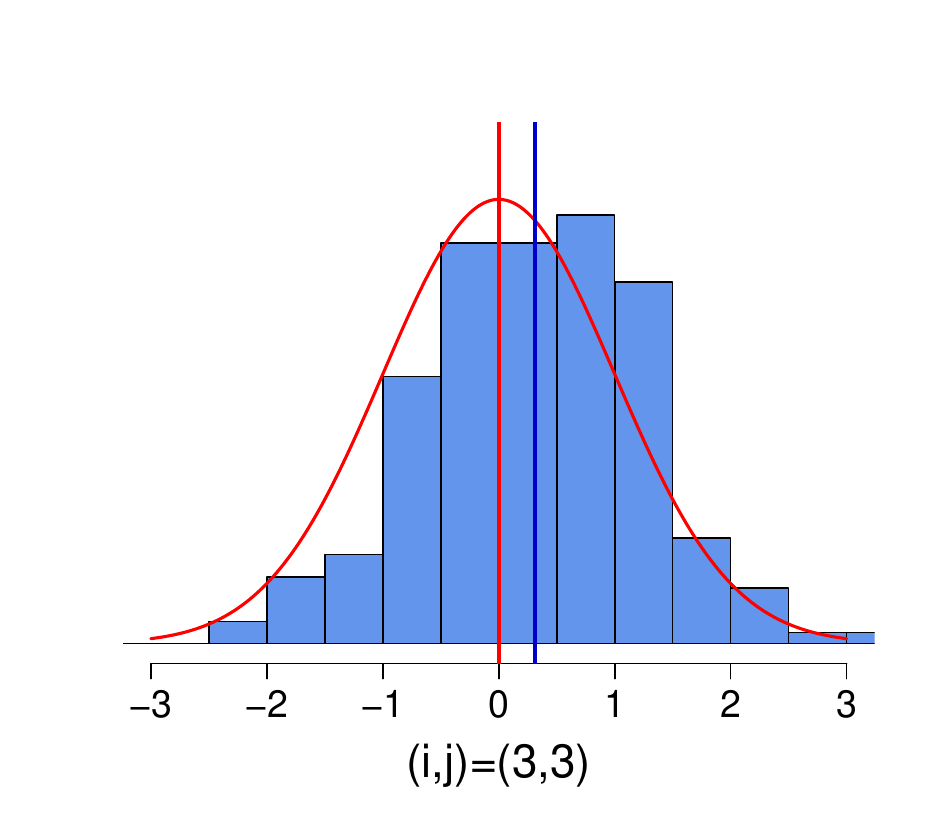}
    \end{minipage}
    \begin{minipage}{0.24\linewidth}
        \centering
        \includegraphics[width=\textwidth]{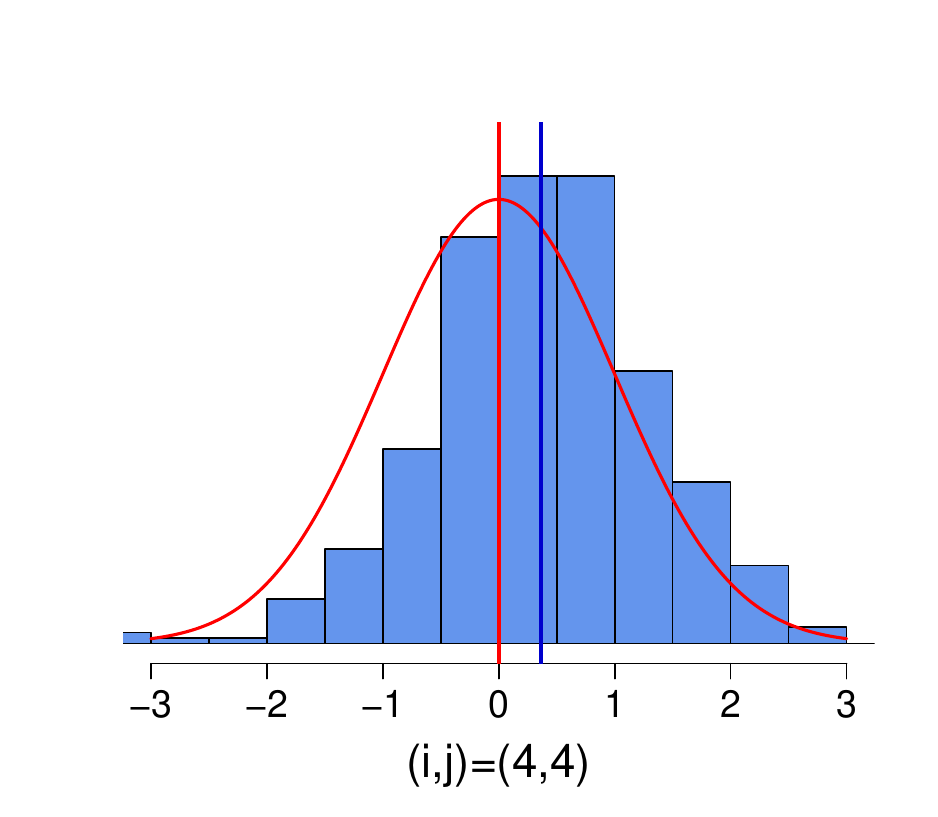}
    \end{minipage}
 \end{minipage}  
     \hspace{1cm}
 \begin{minipage}{0.3\linewidth}
    \begin{minipage}{0.24\linewidth}
        \centering
        \includegraphics[width=\textwidth]{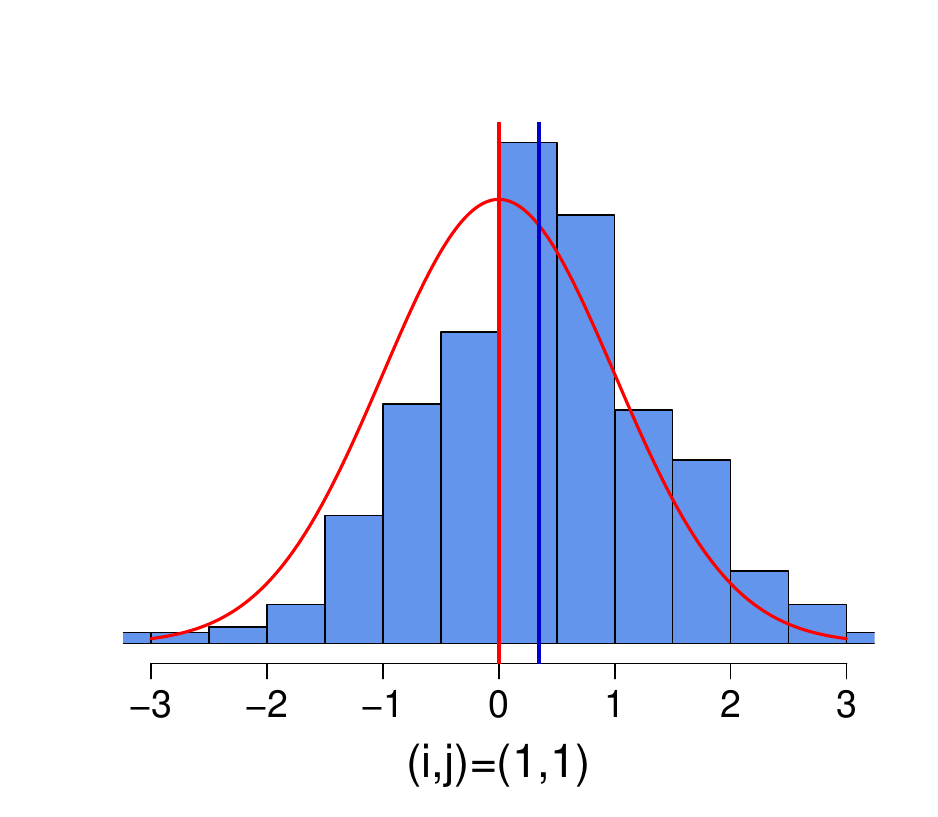}
    \end{minipage}
    \begin{minipage}{0.24\linewidth}
        \centering
        \includegraphics[width=\textwidth]{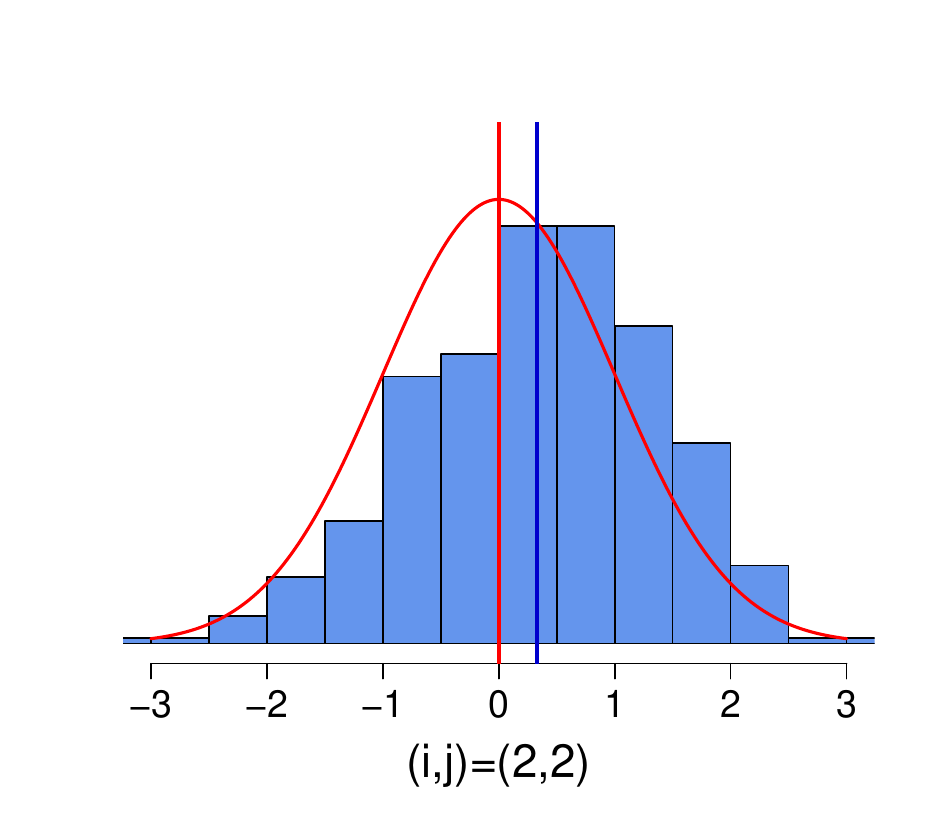}
    \end{minipage}
    \begin{minipage}{0.24\linewidth}
        \centering
        \includegraphics[width=\textwidth]{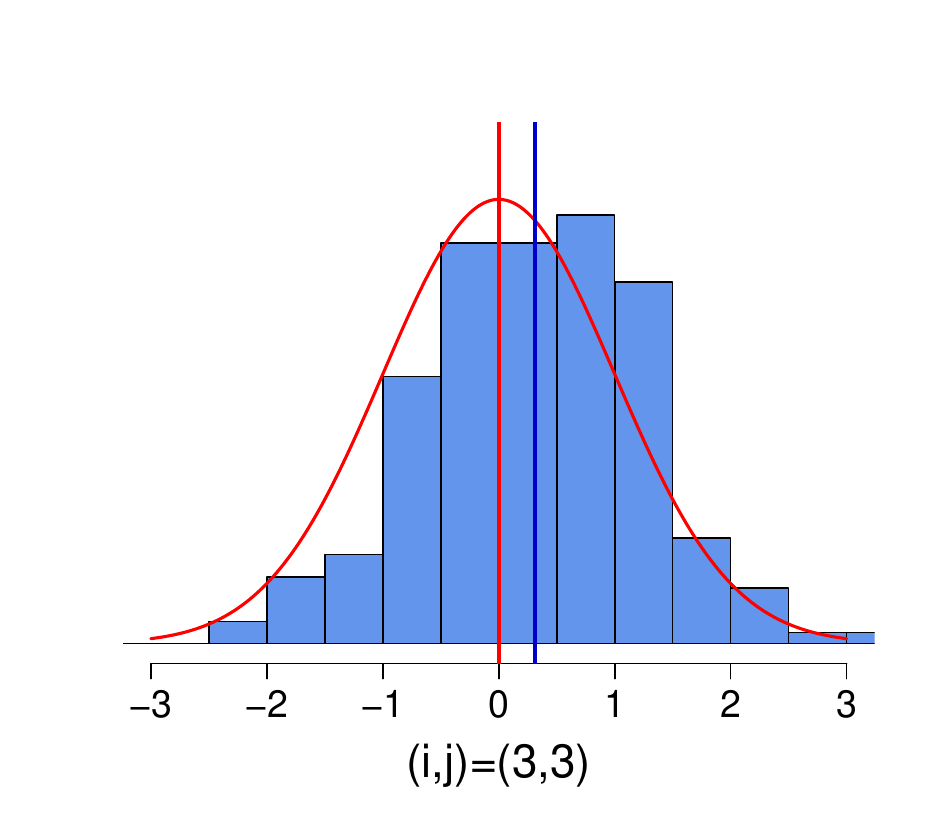}
    \end{minipage}
    \begin{minipage}{0.24\linewidth}
        \centering
        \includegraphics[width=\textwidth]{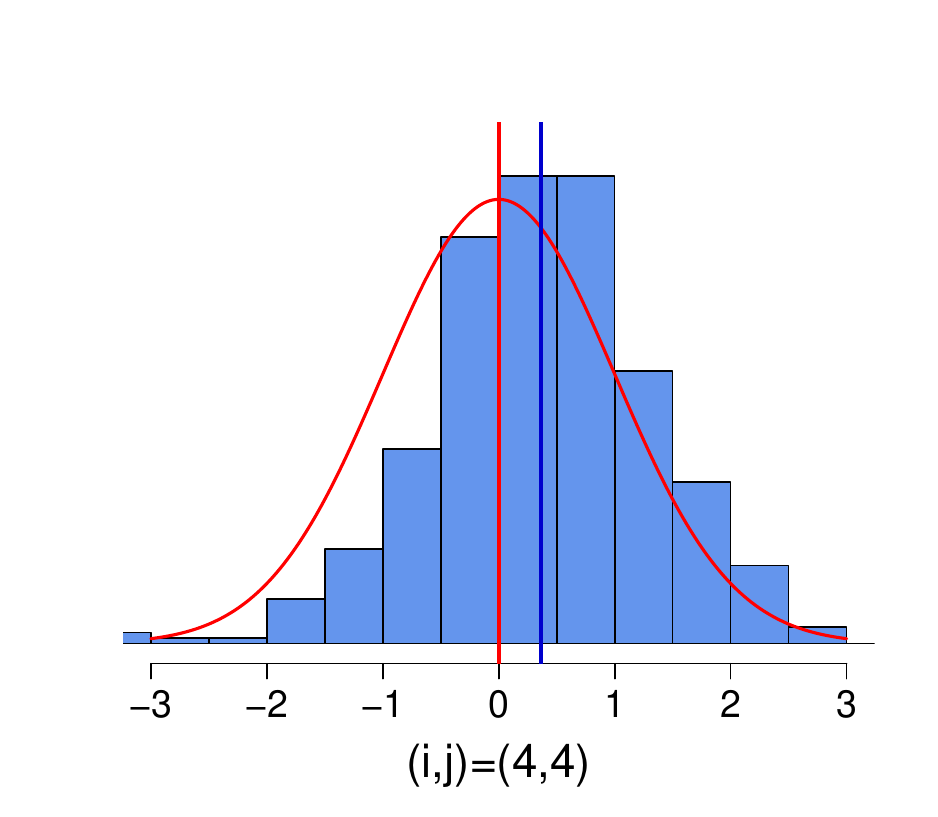}
    \end{minipage}
  \end{minipage}  
    \hspace{1cm}
 \begin{minipage}{0.3\linewidth}
    \begin{minipage}{0.24\linewidth}
        \centering
        \includegraphics[width=\textwidth]{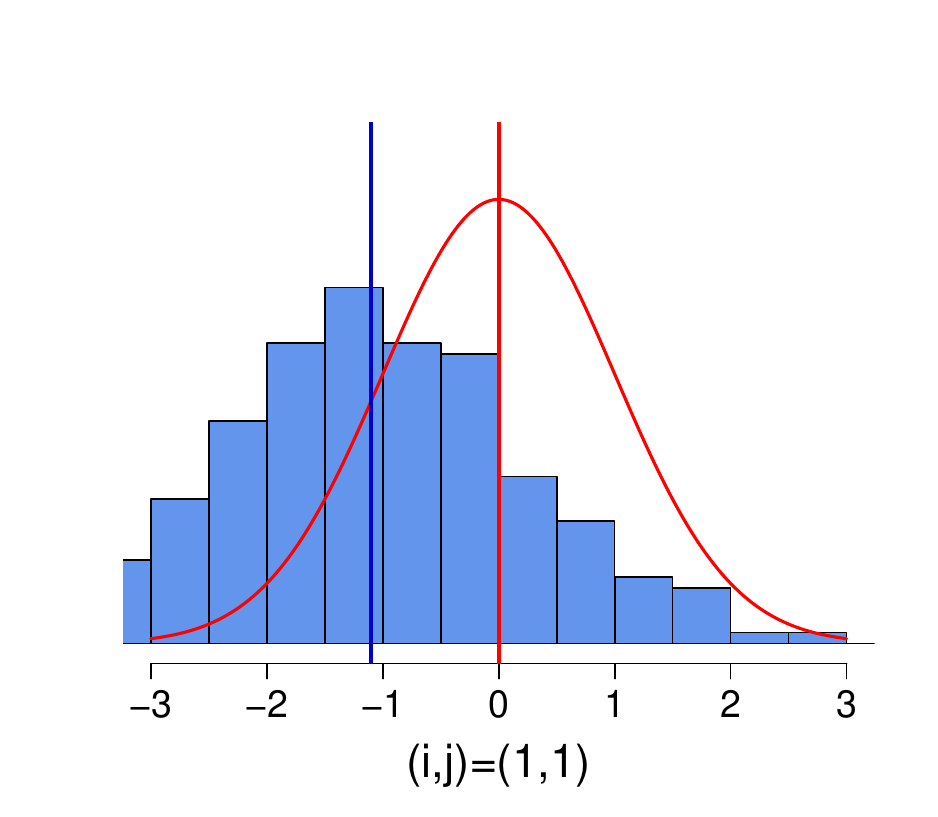}
    \end{minipage}
    \begin{minipage}{0.24\linewidth}
        \centering
        \includegraphics[width=\textwidth]{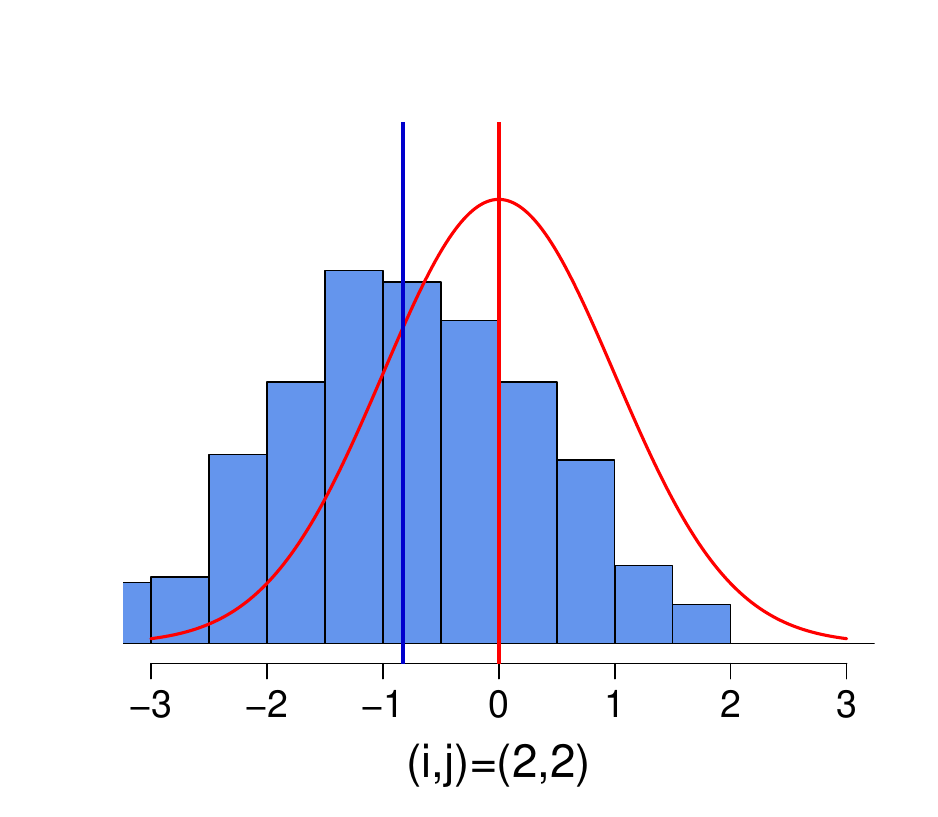}
    \end{minipage}
    \begin{minipage}{0.24\linewidth}
        \centering
        \includegraphics[width=\textwidth]{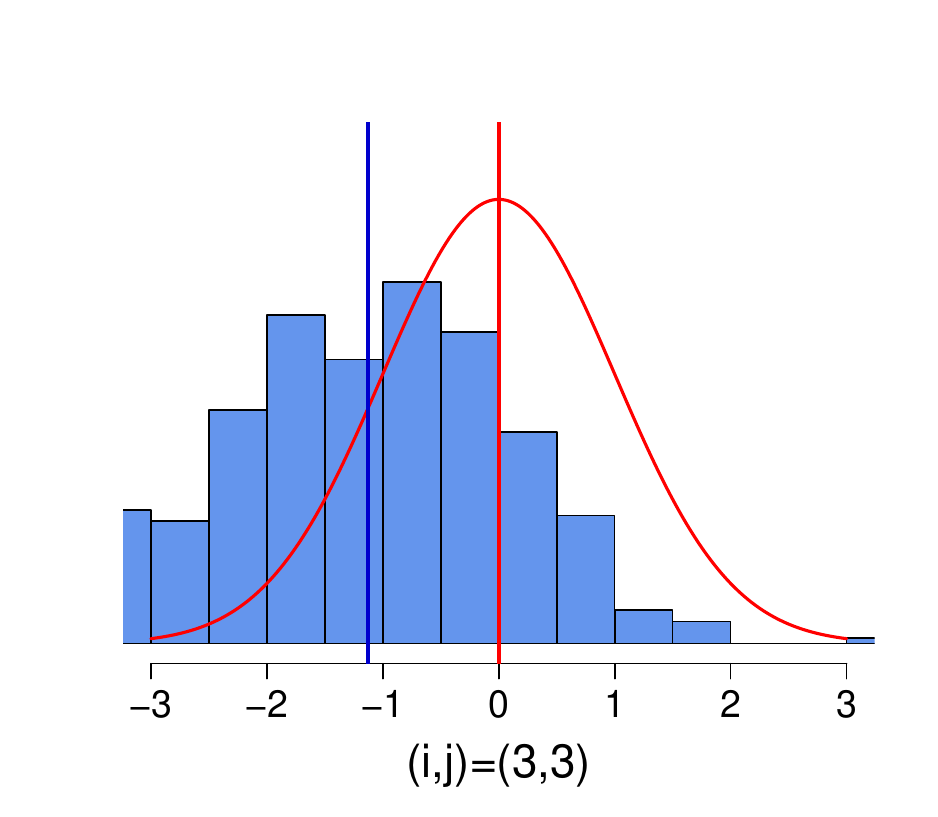}
    \end{minipage}
    \begin{minipage}{0.24\linewidth}
        \centering
        \includegraphics[width=\textwidth]{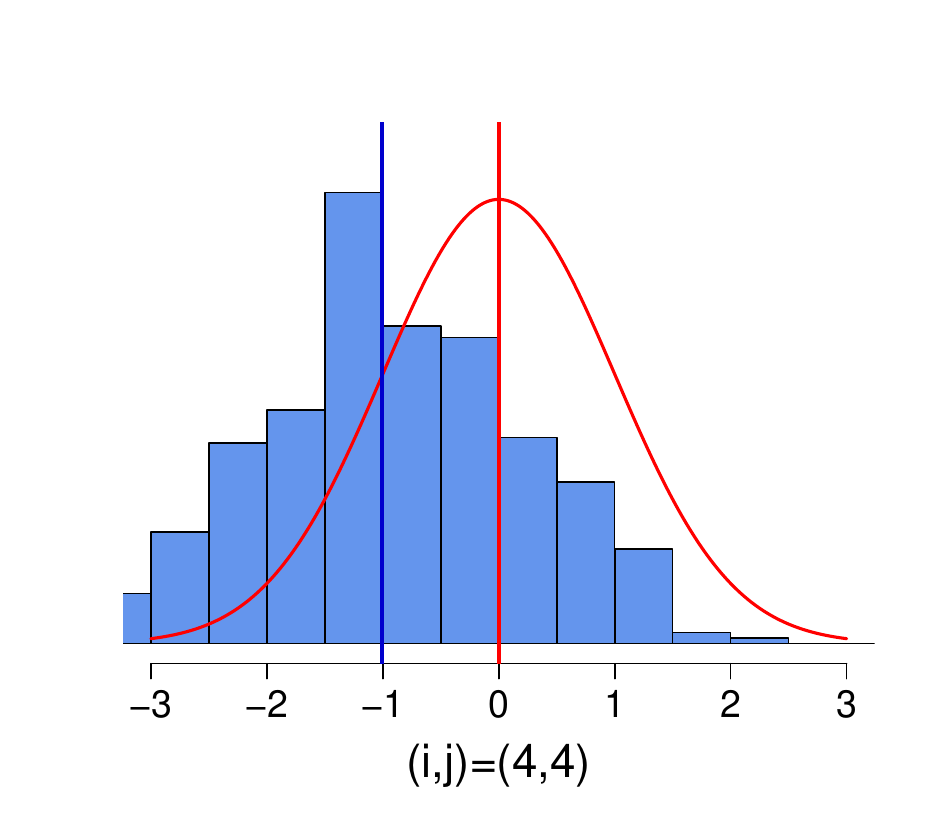}
    \end{minipage}
 \end{minipage}

  \caption*{$n=800, p=200$}
      \vspace{-0.43cm}
 \begin{minipage}{0.3\linewidth}
    \begin{minipage}{0.24\linewidth}
        \centering
        \includegraphics[width=\textwidth]{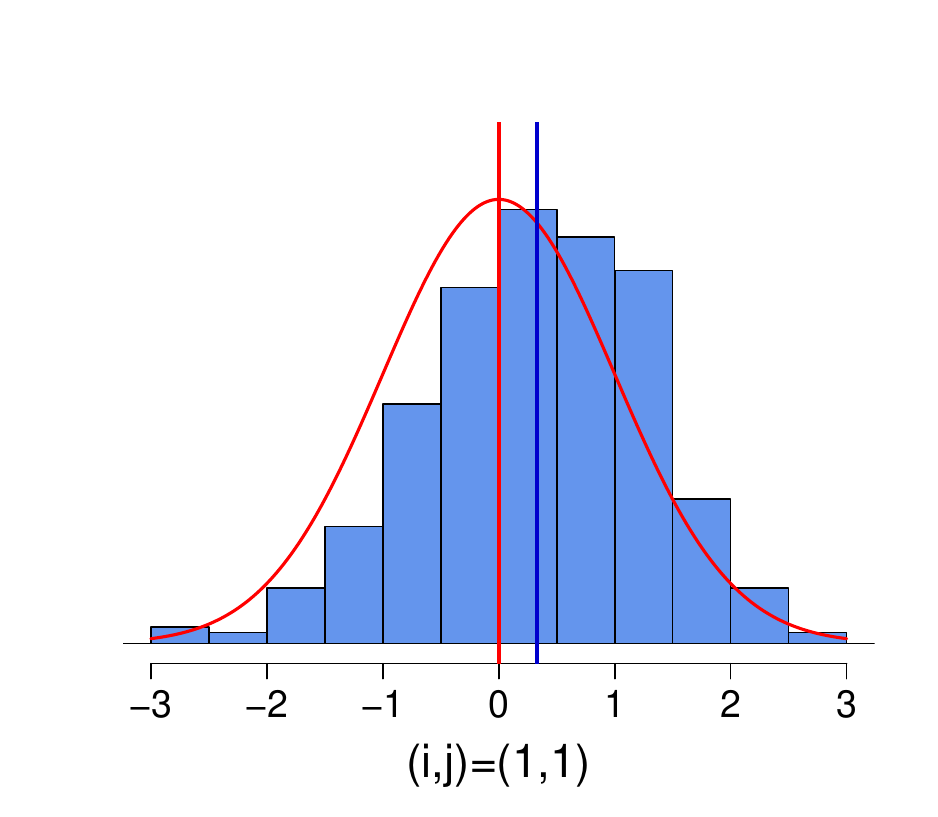}
    \end{minipage}
    \begin{minipage}{0.24\linewidth}
        \centering
        \includegraphics[width=\textwidth]{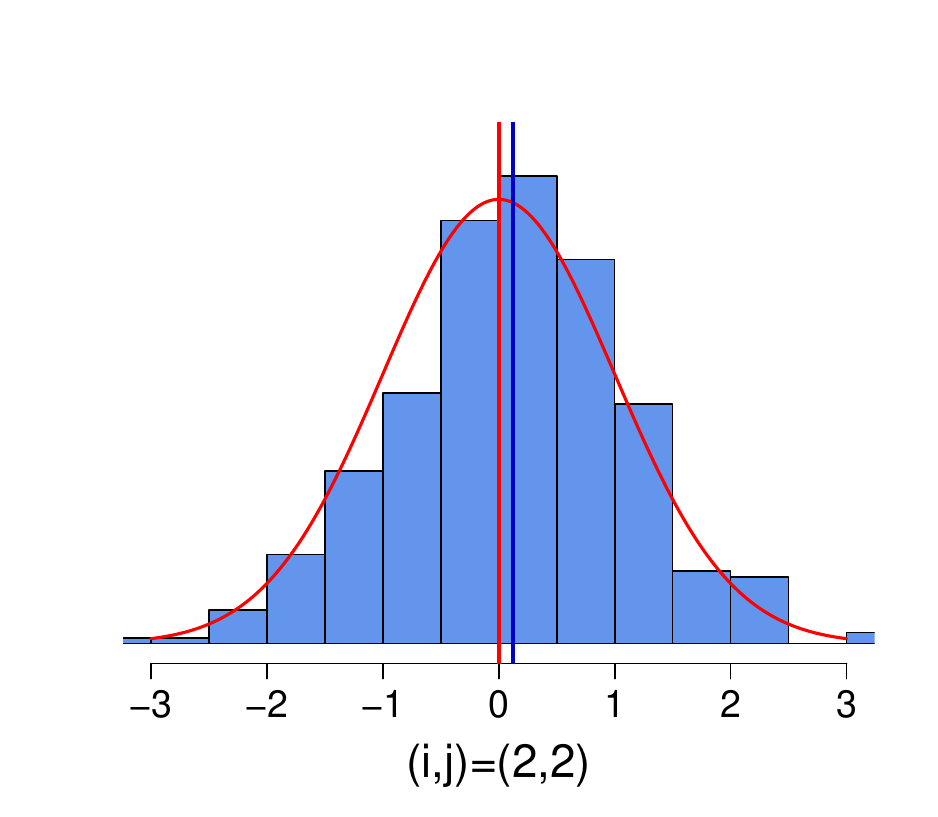}
    \end{minipage}
    \begin{minipage}{0.24\linewidth}
        \centering
        \includegraphics[width=\textwidth]{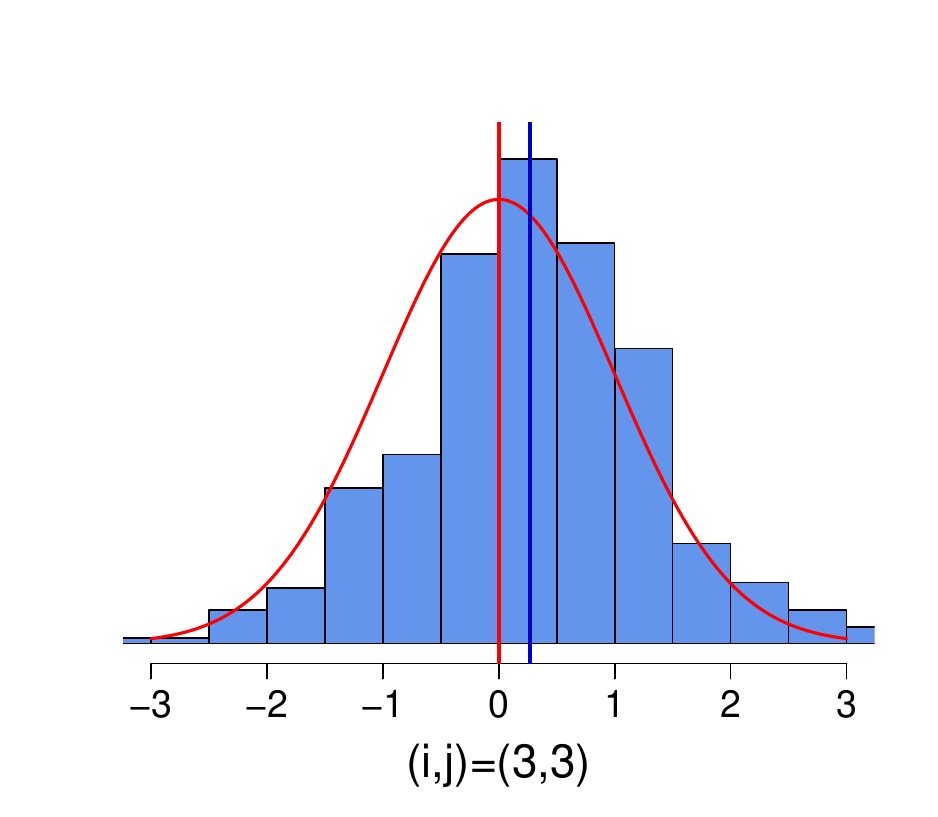}
    \end{minipage}
    \begin{minipage}{0.24\linewidth}
        \centering
        \includegraphics[width=\textwidth]{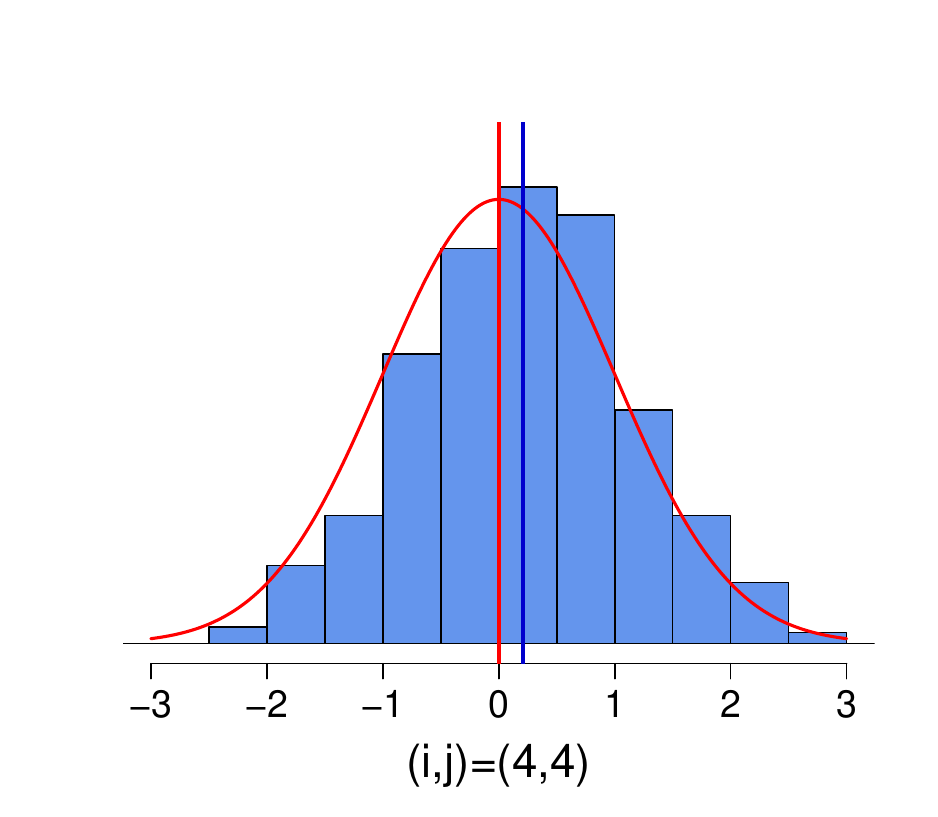}
    \end{minipage}
 \end{minipage} 
     \hspace{1cm}
 \begin{minipage}{0.3\linewidth}
    \begin{minipage}{0.24\linewidth}
        \centering
        \includegraphics[width=\textwidth]{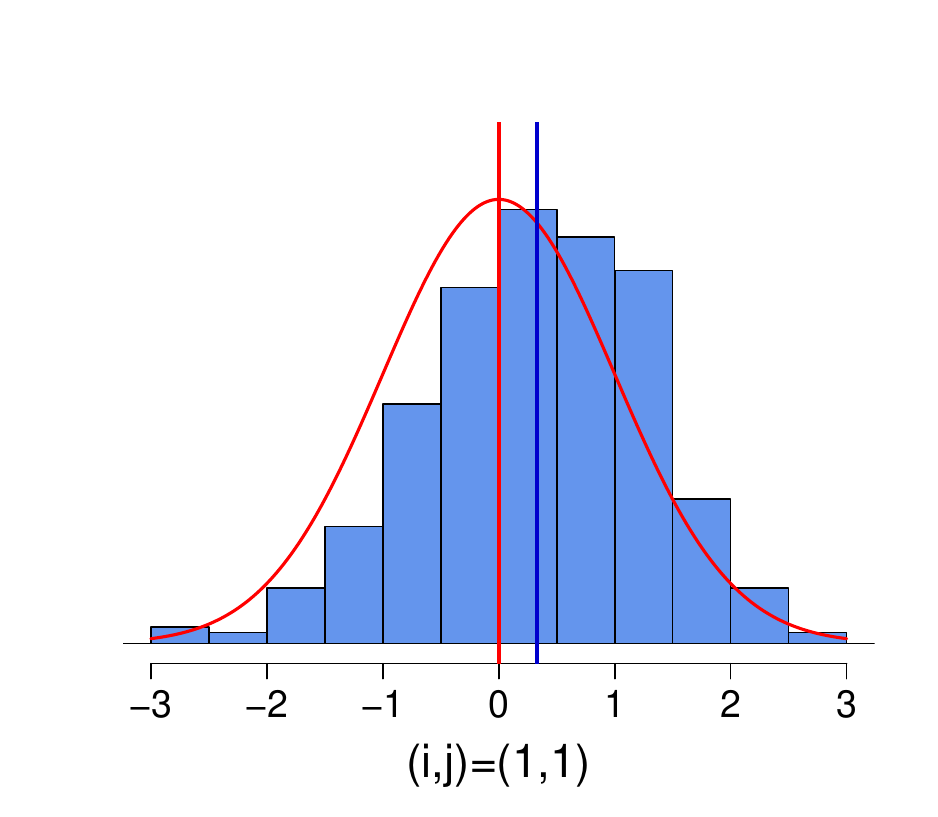}
    \end{minipage}
    \begin{minipage}{0.24\linewidth}
        \centering
        \includegraphics[width=\textwidth]{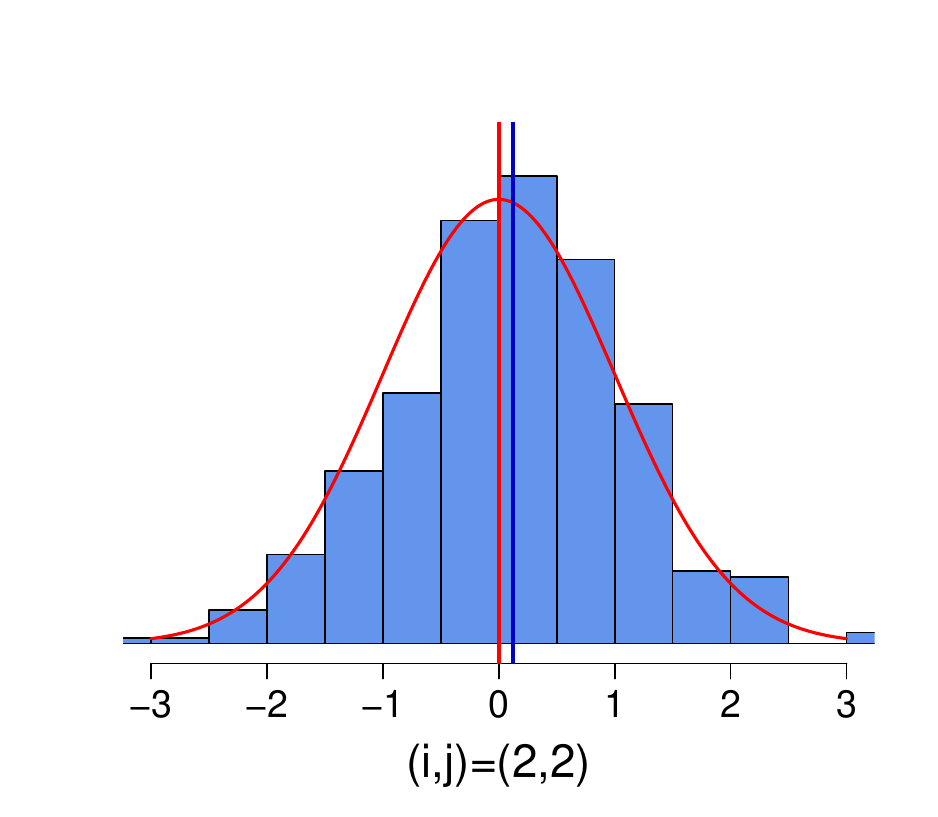}
    \end{minipage}
    \begin{minipage}{0.24\linewidth}
        \centering
        \includegraphics[width=\textwidth]{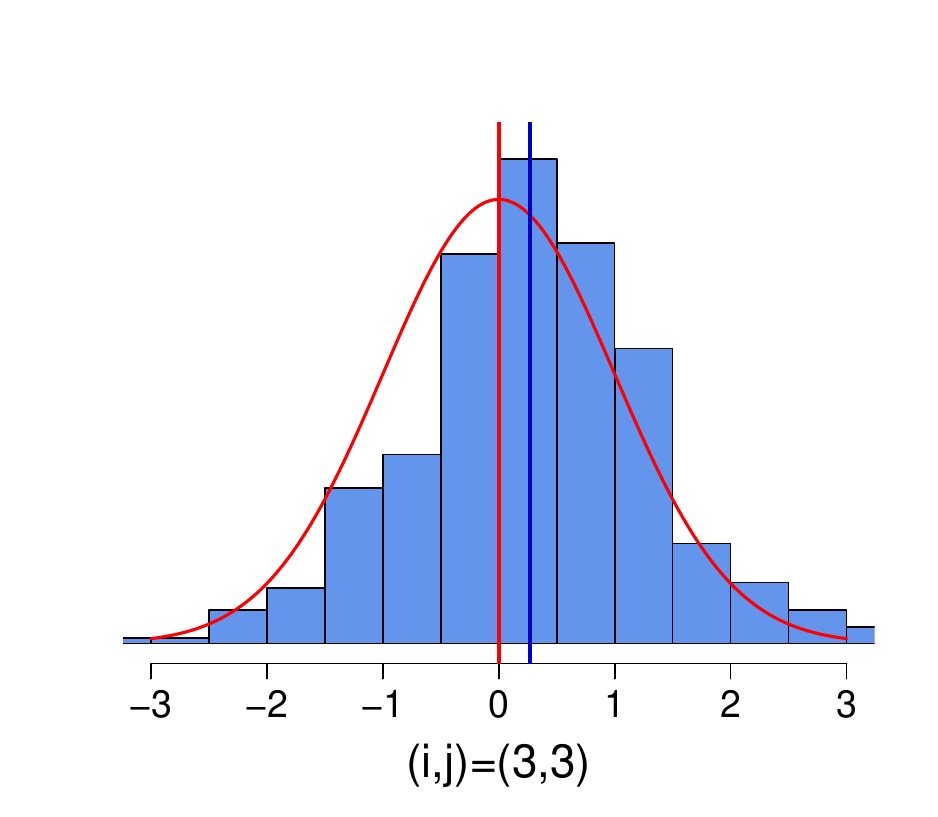}
    \end{minipage}
    \begin{minipage}{0.24\linewidth}
        \centering
        \includegraphics[width=\textwidth]{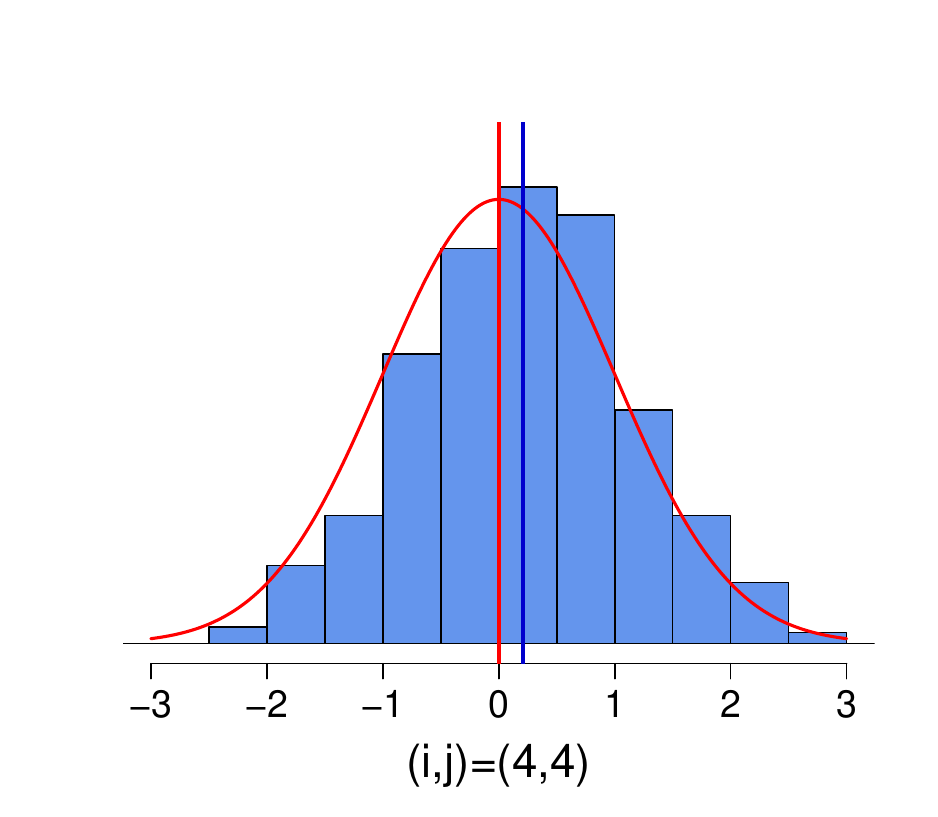}
    \end{minipage}
 \end{minipage}   
      \hspace{1cm}
 \begin{minipage}{0.3\linewidth}
    \begin{minipage}{0.24\linewidth}
        \centering
        \includegraphics[width=\textwidth]{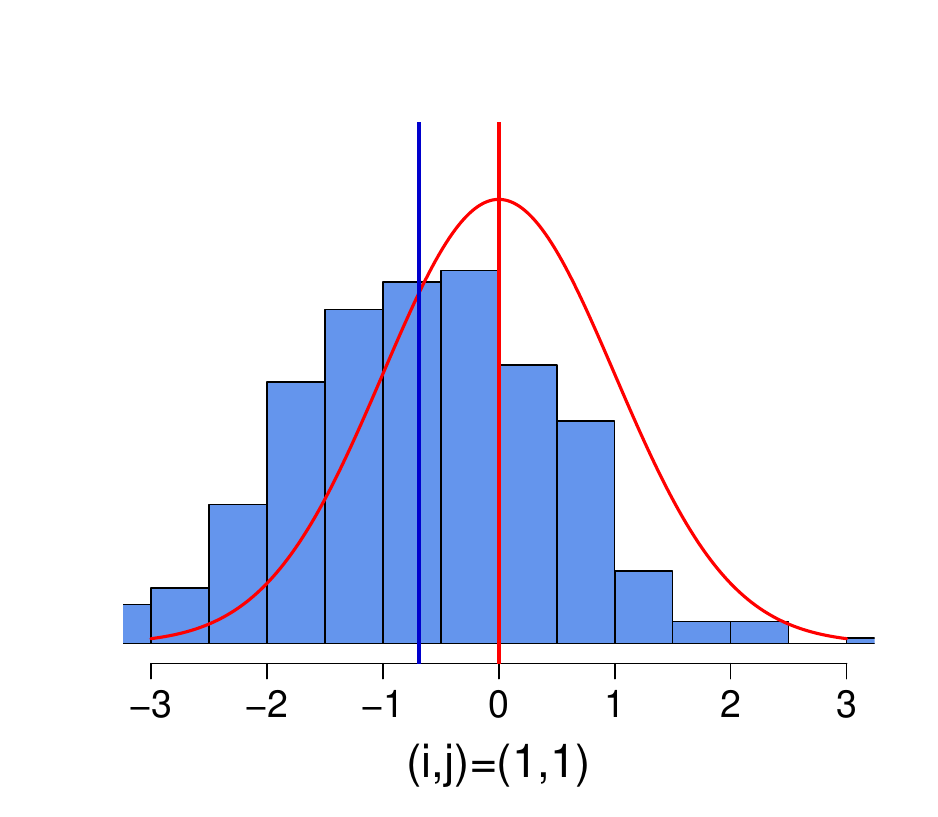}
    \end{minipage}
    \begin{minipage}{0.24\linewidth}
        \centering
        \includegraphics[width=\textwidth]{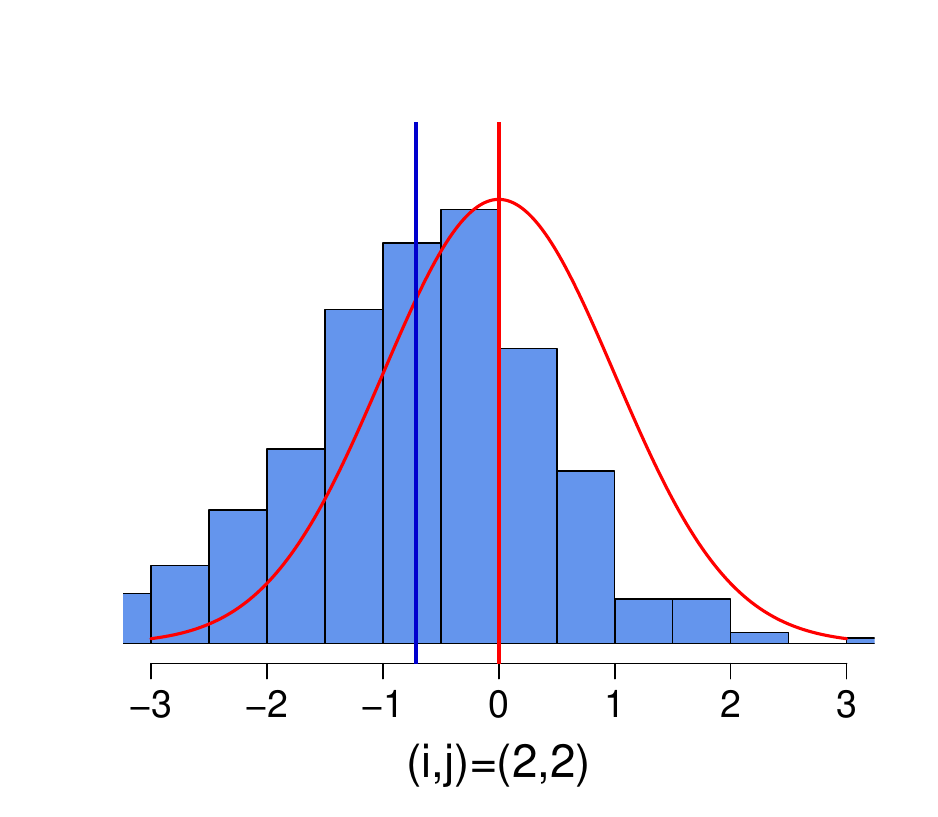}
    \end{minipage}
    \begin{minipage}{0.24\linewidth}
        \centering
        \includegraphics[width=\textwidth]{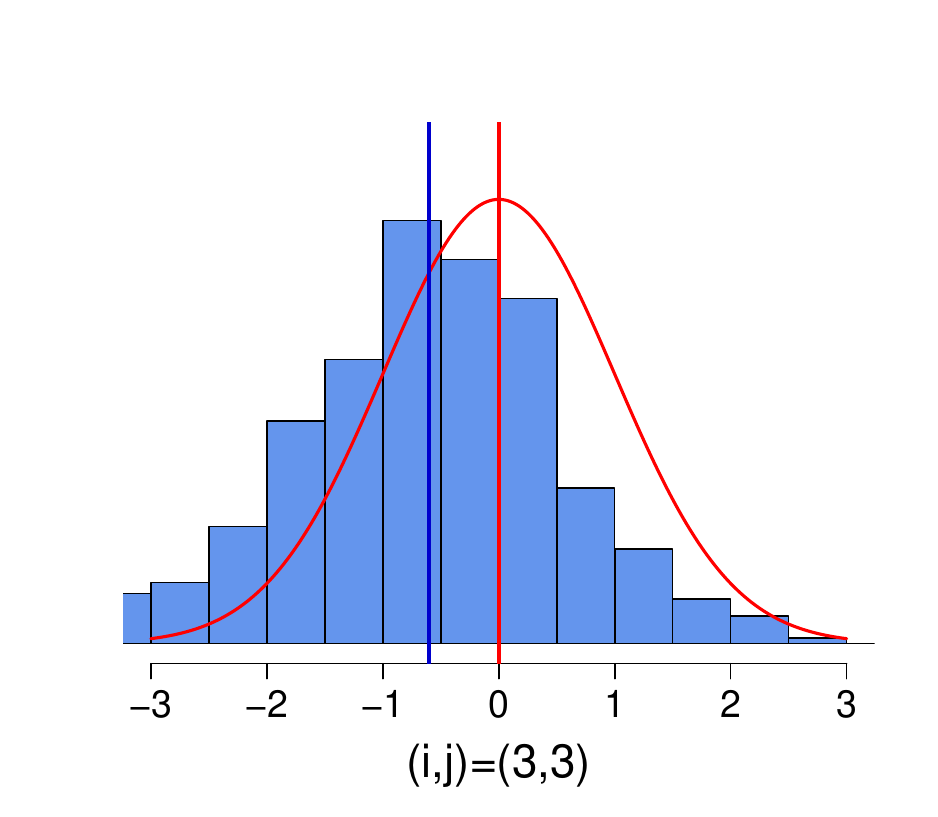}
    \end{minipage}
    \begin{minipage}{0.24\linewidth}
        \centering
        \includegraphics[width=\textwidth]{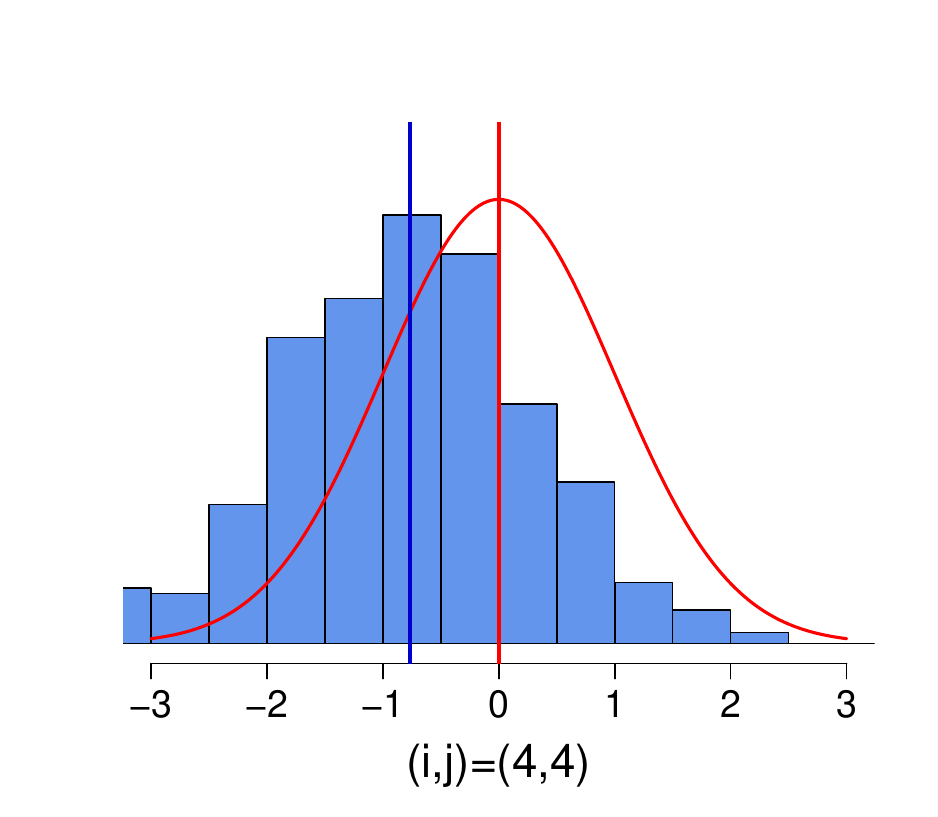}
    \end{minipage}
     \end{minipage}   
     
 \caption*{$n=200, p=400$}
     \vspace{-0.43cm}
 \begin{minipage}{0.3\linewidth}
    \begin{minipage}{0.24\linewidth}
        \centering
        \includegraphics[width=\textwidth]{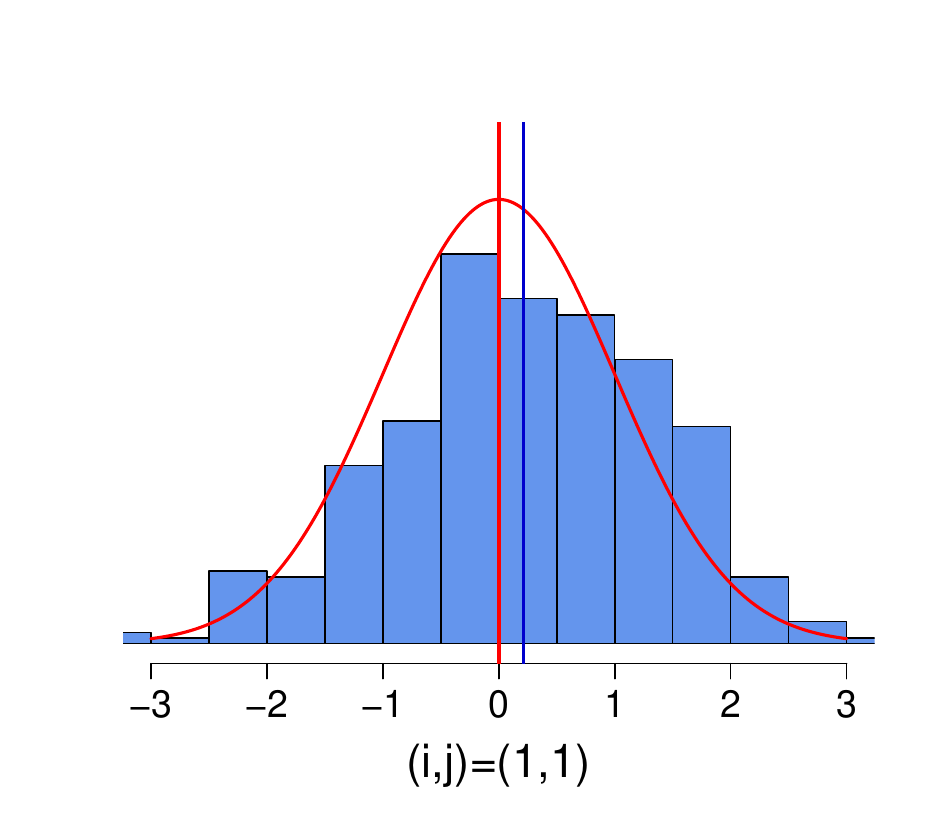}
    \end{minipage}
    \begin{minipage}{0.24\linewidth}
        \centering
        \includegraphics[width=\textwidth]{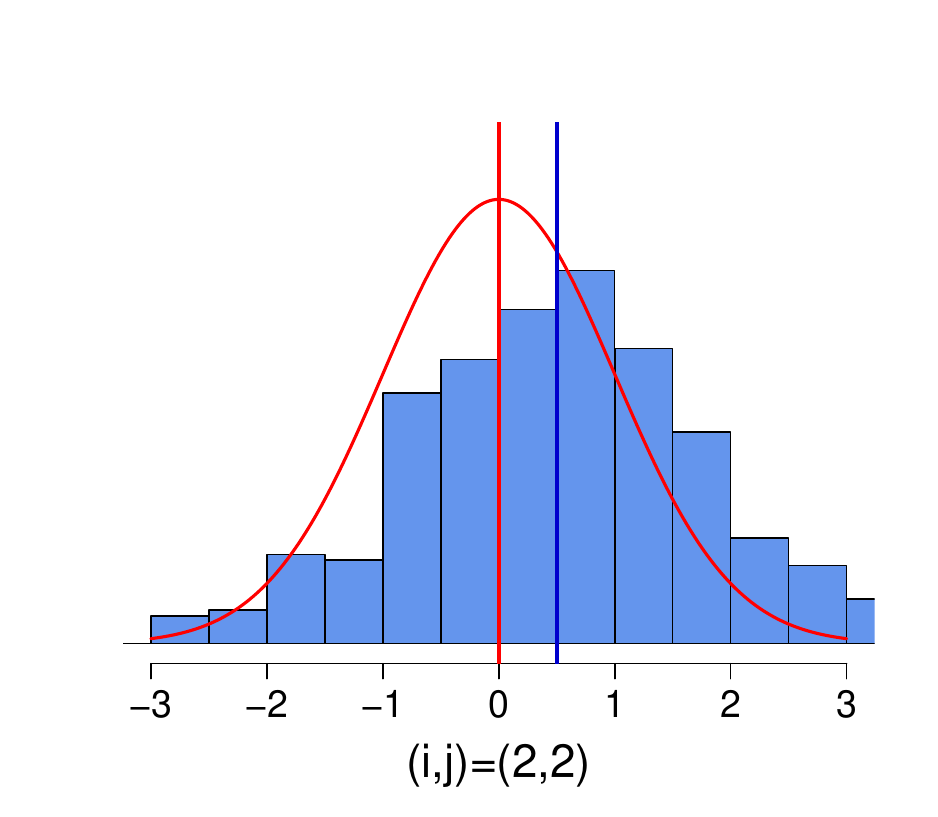}
    \end{minipage}
    \begin{minipage}{0.24\linewidth}
        \centering
        \includegraphics[width=\textwidth]{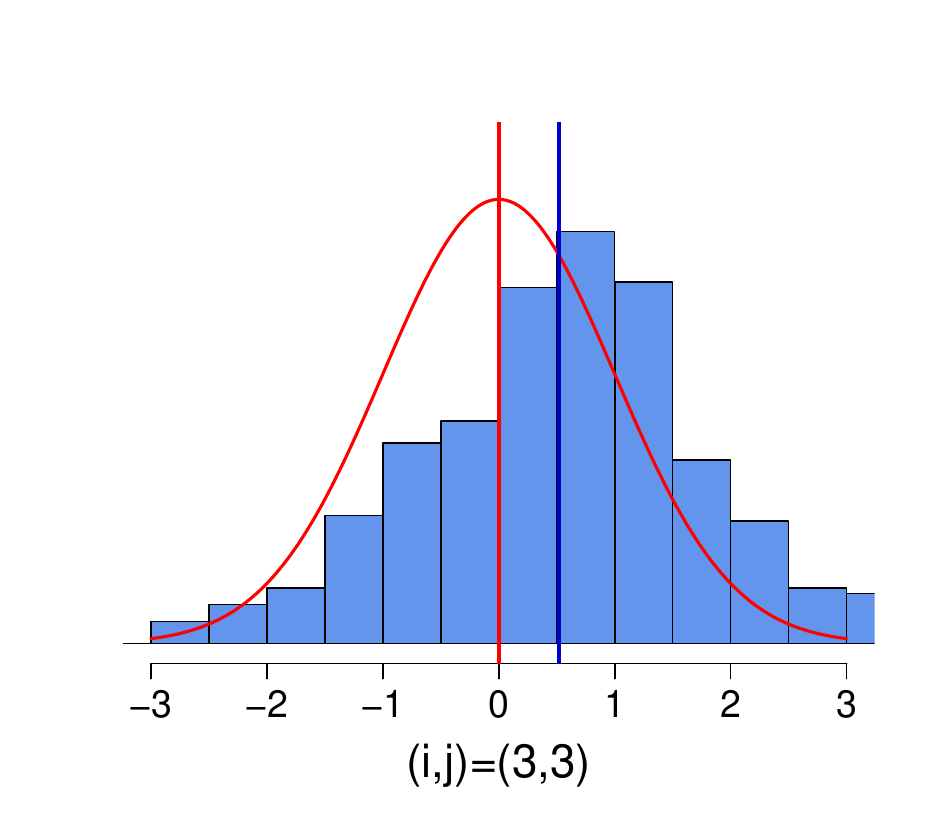}
    \end{minipage}
    \begin{minipage}{0.24\linewidth}
        \centering
        \includegraphics[width=\textwidth]{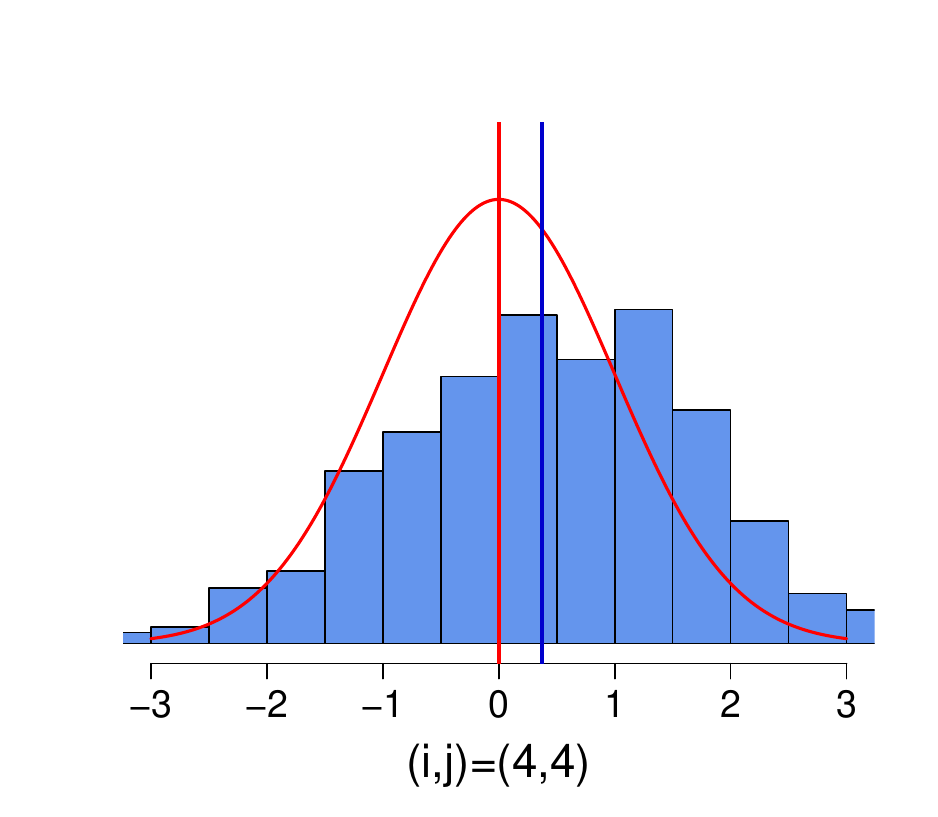}
    \end{minipage}
 \end{minipage}
 \hspace{1cm}
 \begin{minipage}{0.3\linewidth}
    \begin{minipage}{0.24\linewidth}
        \centering
        \includegraphics[width=\textwidth]{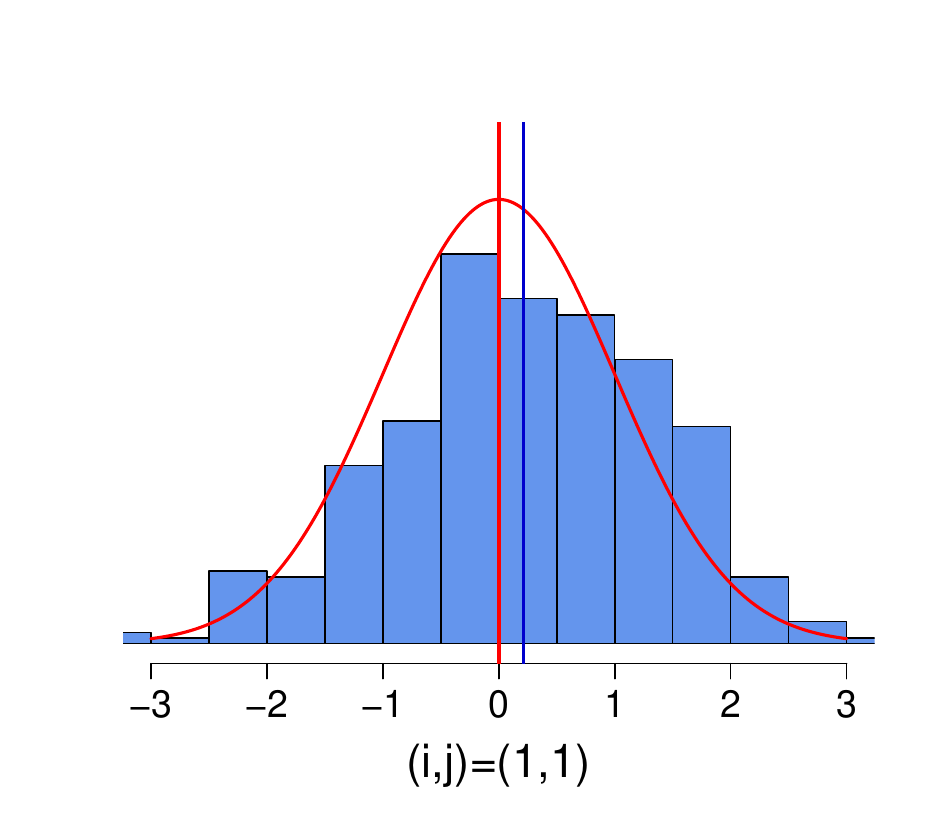}
    \end{minipage}
    \begin{minipage}{0.24\linewidth}
        \centering
        \includegraphics[width=\textwidth]{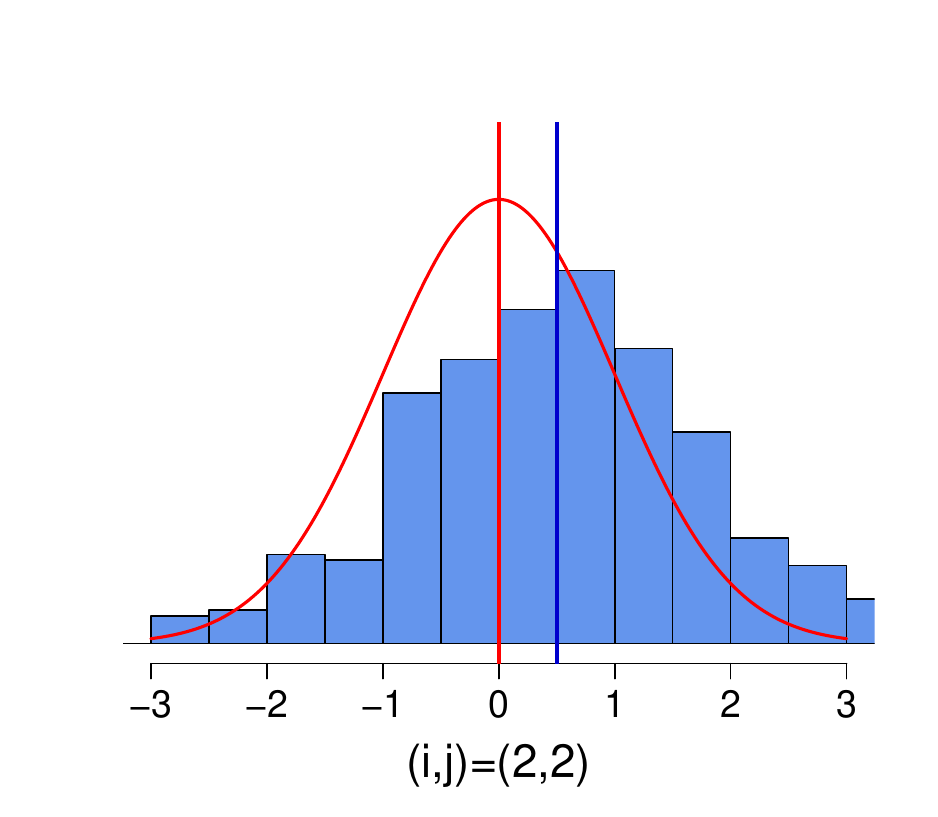}
    \end{minipage}
    \begin{minipage}{0.24\linewidth}
        \centering
        \includegraphics[width=\textwidth]{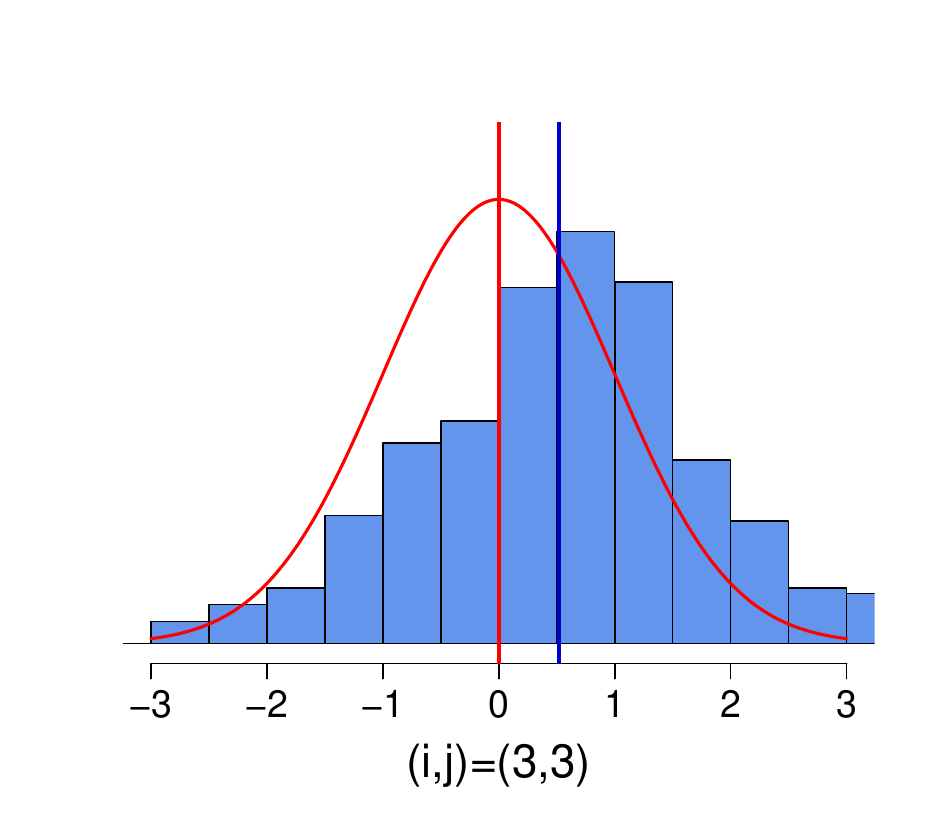}
    \end{minipage}
    \begin{minipage}{0.24\linewidth}
        \centering
        \includegraphics[width=\textwidth]{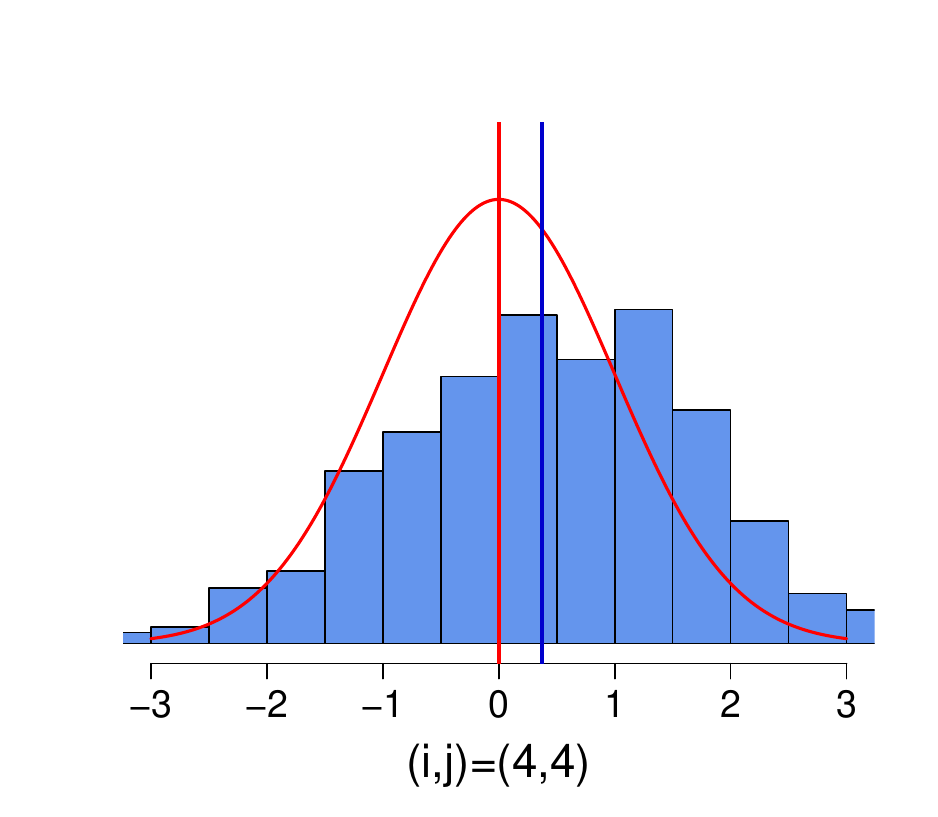}
    \end{minipage}    
 \end{minipage}
  \hspace{1cm}
 \begin{minipage}{0.3\linewidth}
     \begin{minipage}{0.24\linewidth}
        \centering
        \includegraphics[width=\textwidth]{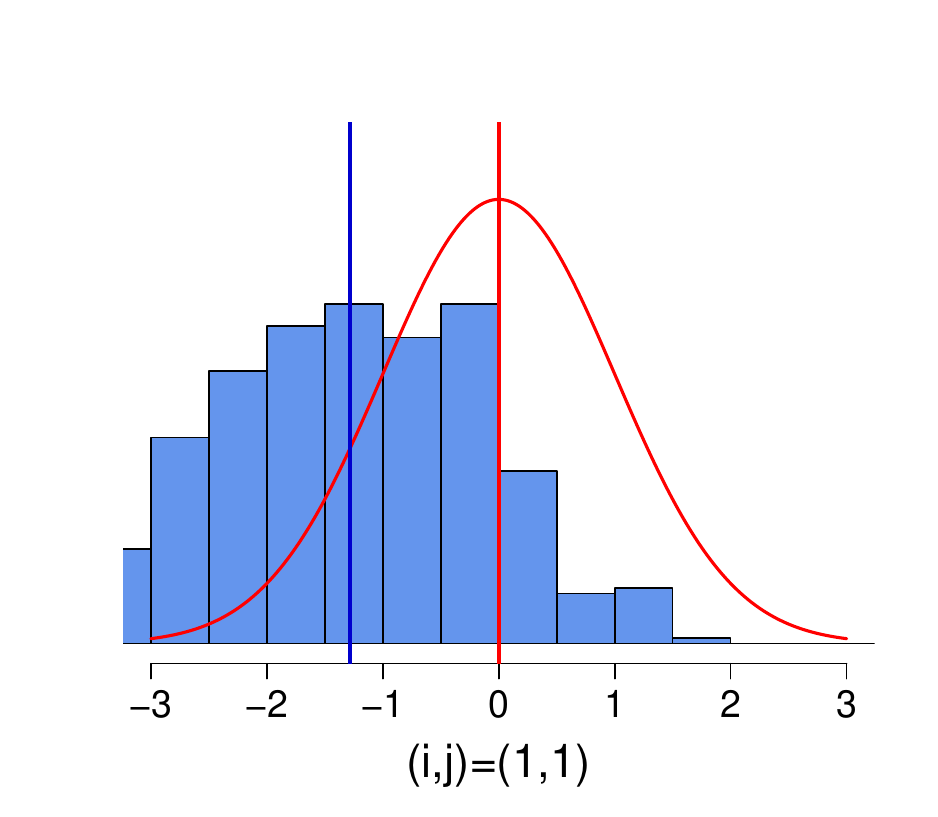}
    \end{minipage}
    \begin{minipage}{0.24\linewidth}
        \centering
        \includegraphics[width=\textwidth]{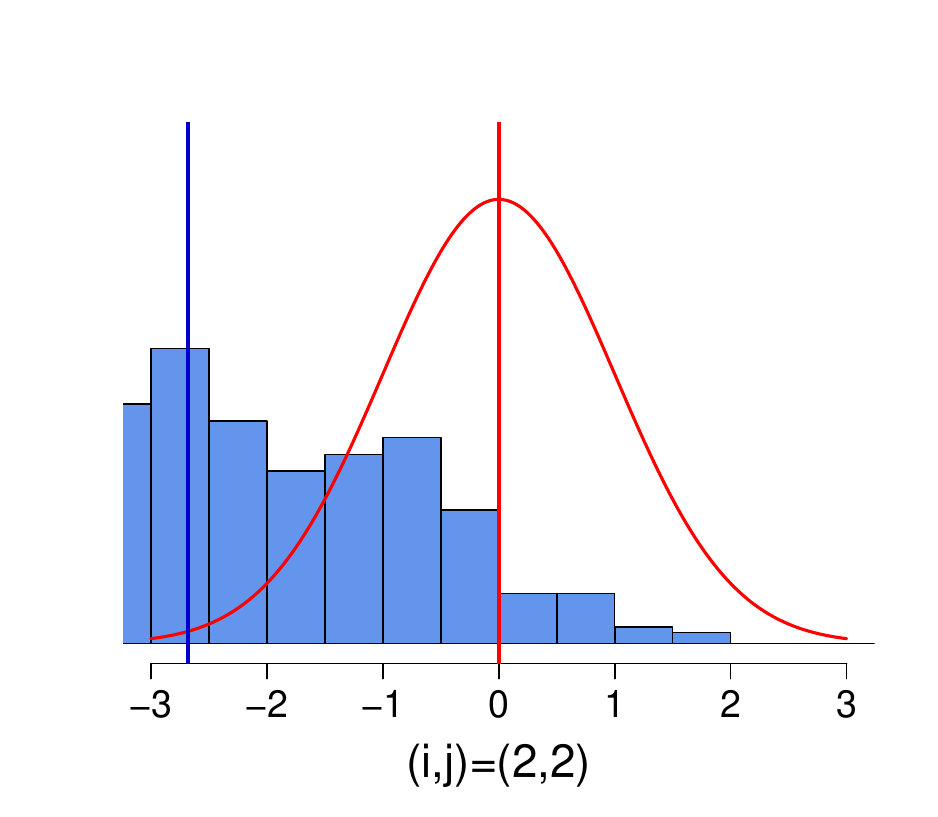}
    \end{minipage}
    \begin{minipage}{0.24\linewidth}
        \centering
        \includegraphics[width=\textwidth]{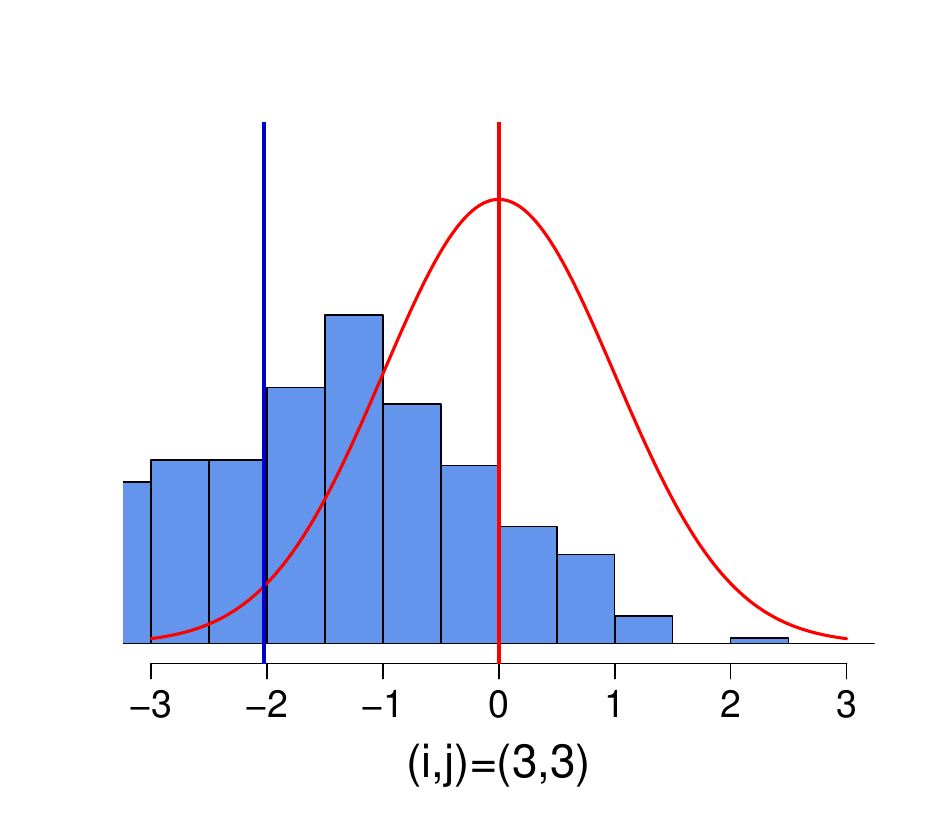}
    \end{minipage}
    \begin{minipage}{0.24\linewidth}
        \centering
        \includegraphics[width=\textwidth]{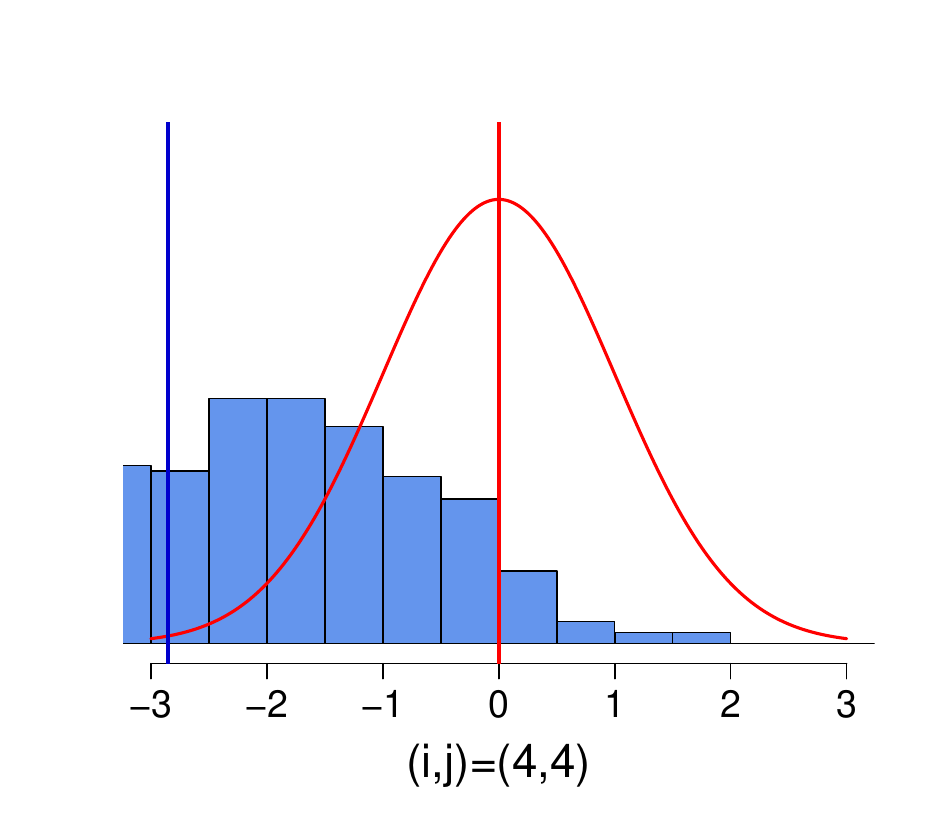}
    \end{minipage}
 \end{minipage}

  \caption*{$n=400, p=400$}
      \vspace{-0.43cm}
 \begin{minipage}{0.3\linewidth}
    \begin{minipage}{0.24\linewidth}
        \centering
        \includegraphics[width=\textwidth]{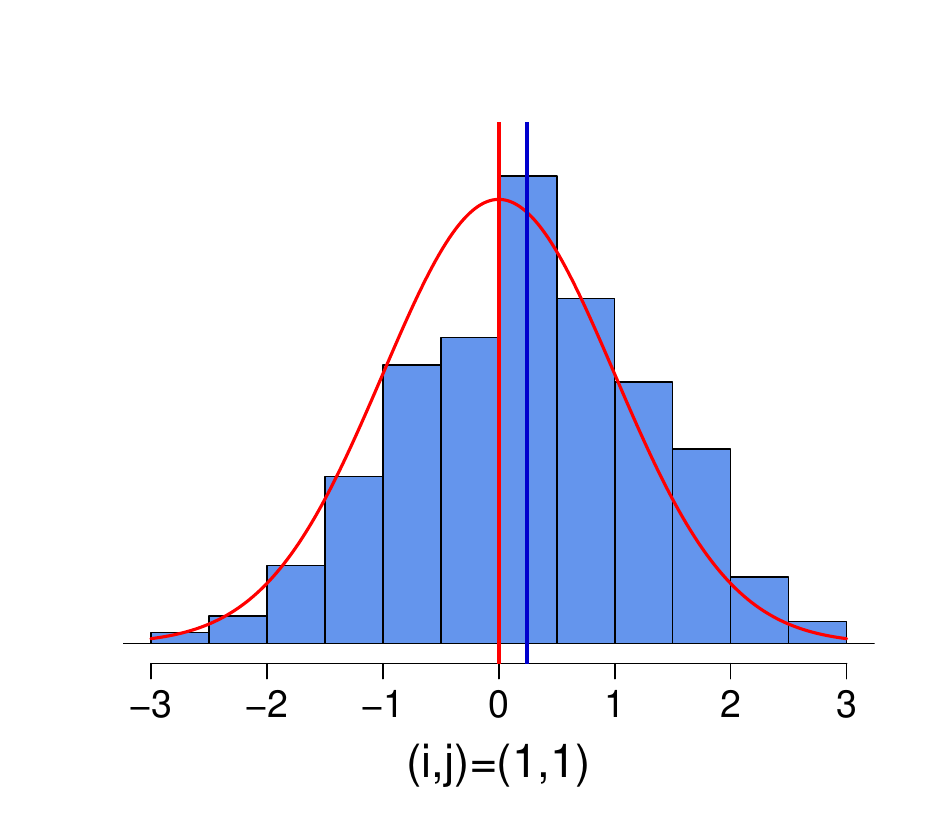}
    \end{minipage}
    \begin{minipage}{0.24\linewidth}
        \centering
        \includegraphics[width=\textwidth]{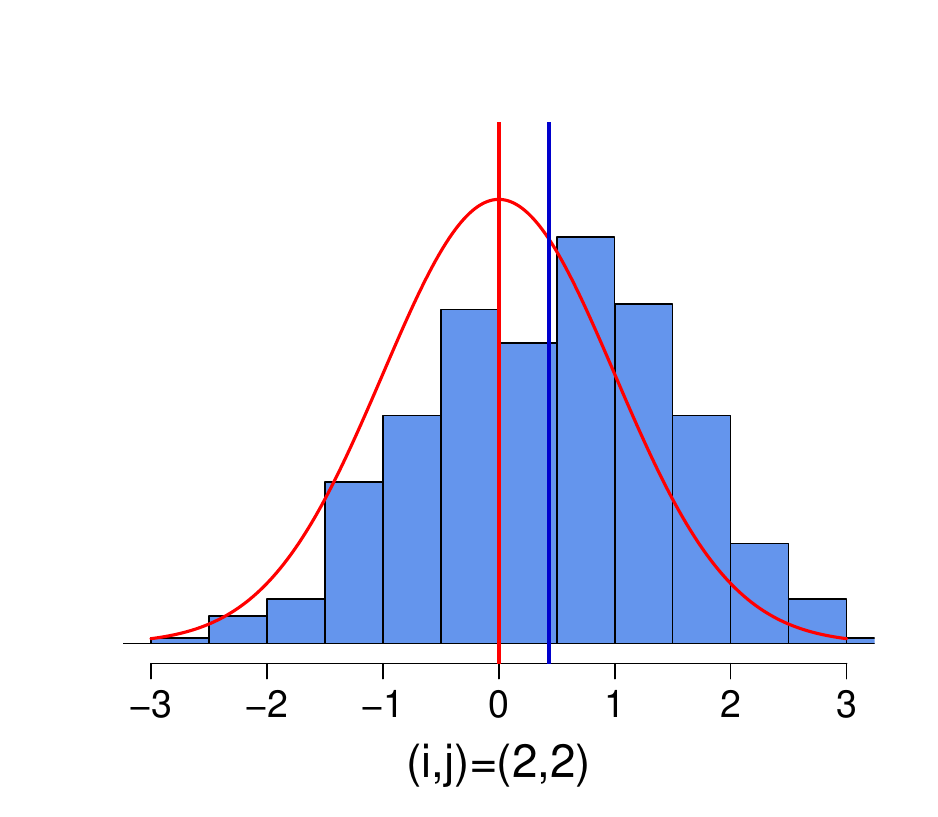}
    \end{minipage}
    \begin{minipage}{0.24\linewidth}
        \centering
        \includegraphics[width=\textwidth]{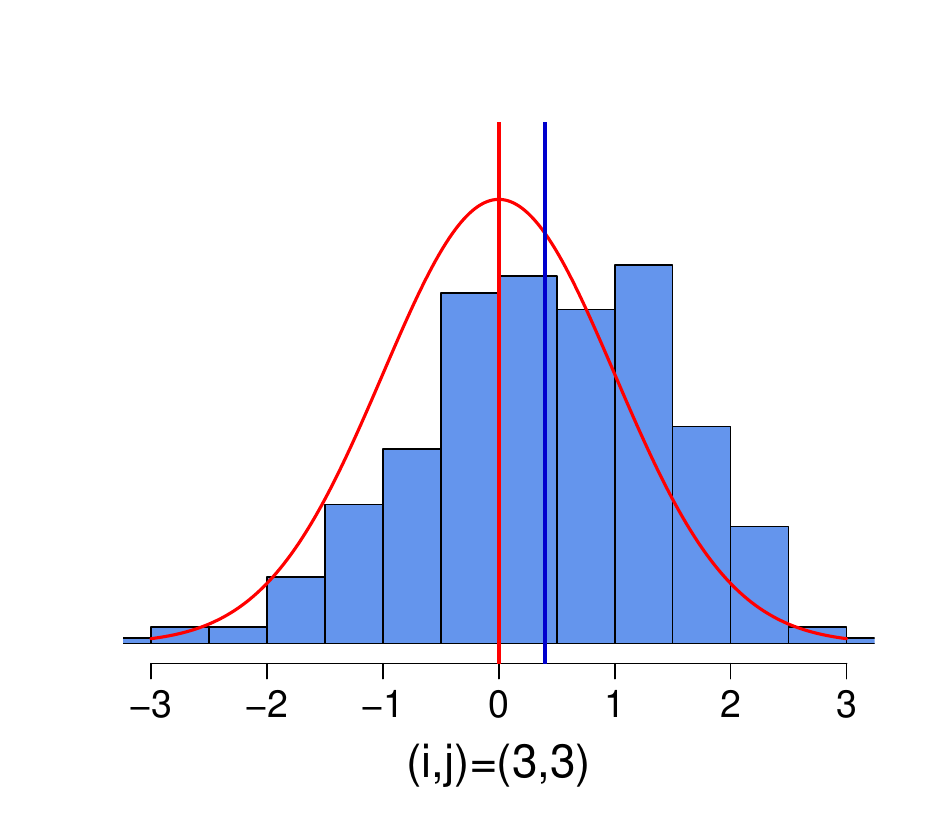}
    \end{minipage}
    \begin{minipage}{0.24\linewidth}
        \centering
        \includegraphics[width=\textwidth]{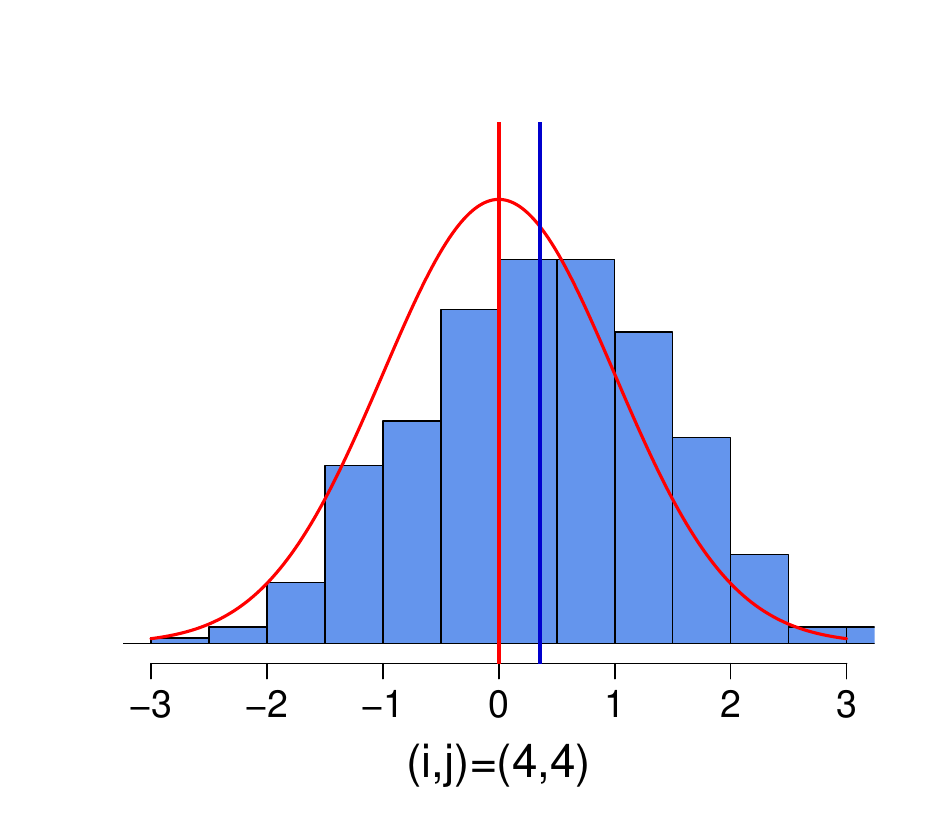}
    \end{minipage}
 \end{minipage}  
     \hspace{1cm}
 \begin{minipage}{0.3\linewidth}
    \begin{minipage}{0.24\linewidth}
        \centering
        \includegraphics[width=\textwidth]{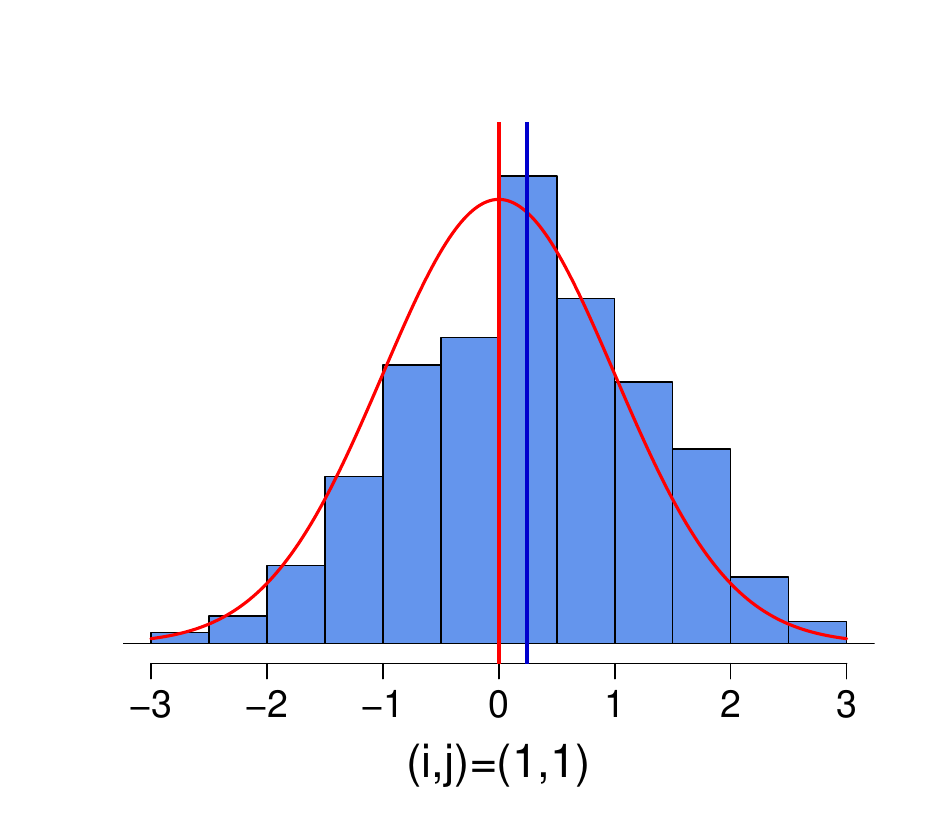}
    \end{minipage}
    \begin{minipage}{0.24\linewidth}
        \centering
        \includegraphics[width=\textwidth]{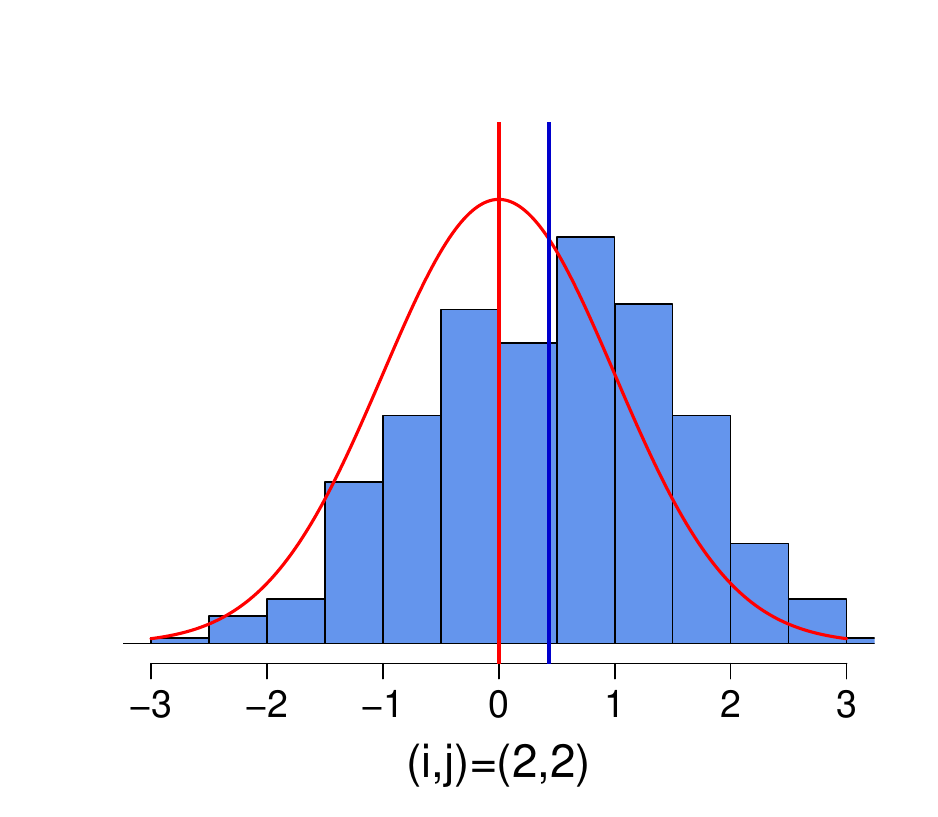}
    \end{minipage}
    \begin{minipage}{0.24\linewidth}
        \centering
        \includegraphics[width=\textwidth]{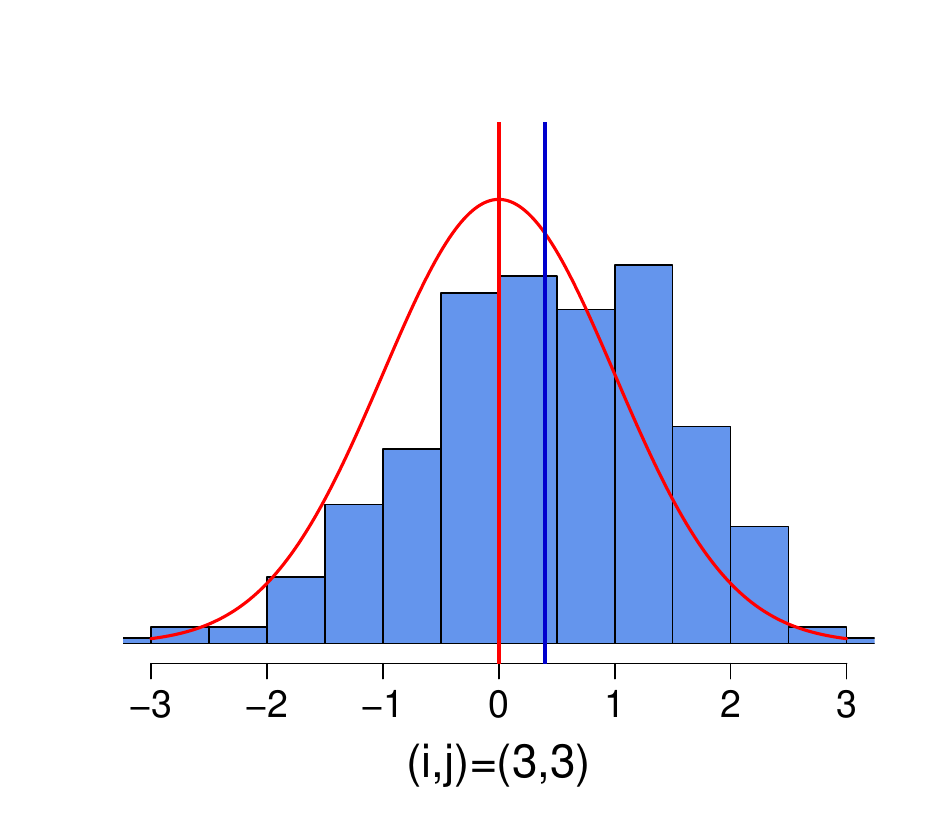}
    \end{minipage}
    \begin{minipage}{0.24\linewidth}
        \centering
        \includegraphics[width=\textwidth]{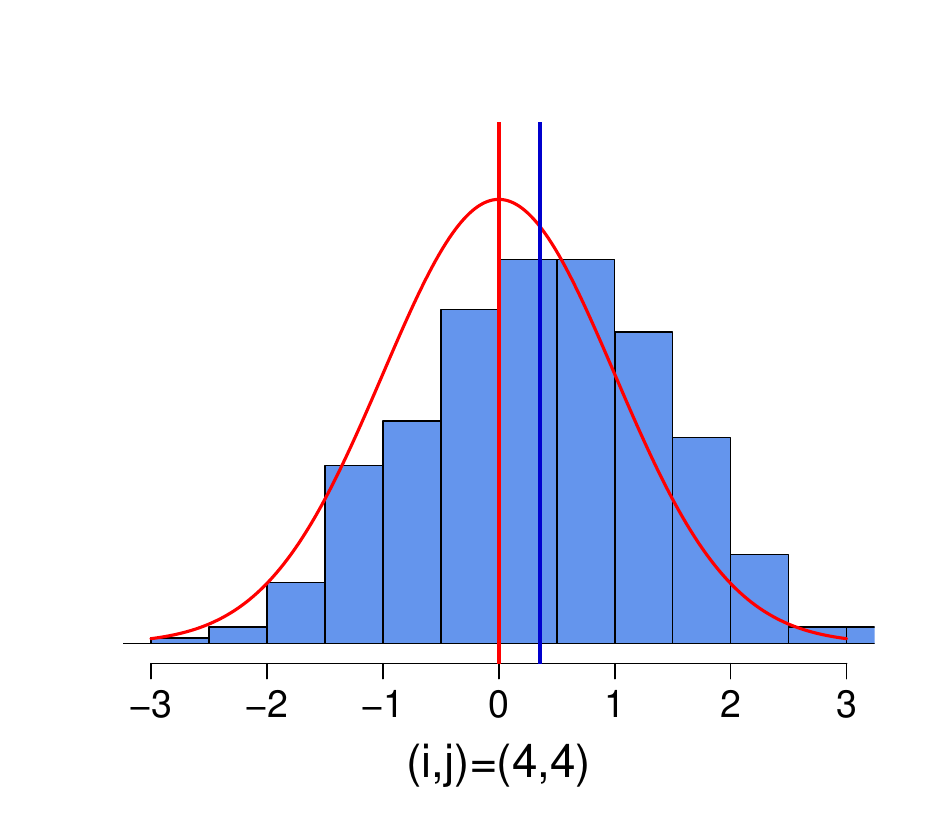}
    \end{minipage}
  \end{minipage}  
    \hspace{1cm}
 \begin{minipage}{0.3\linewidth}
    \begin{minipage}{0.24\linewidth}
        \centering
        \includegraphics[width=\textwidth]{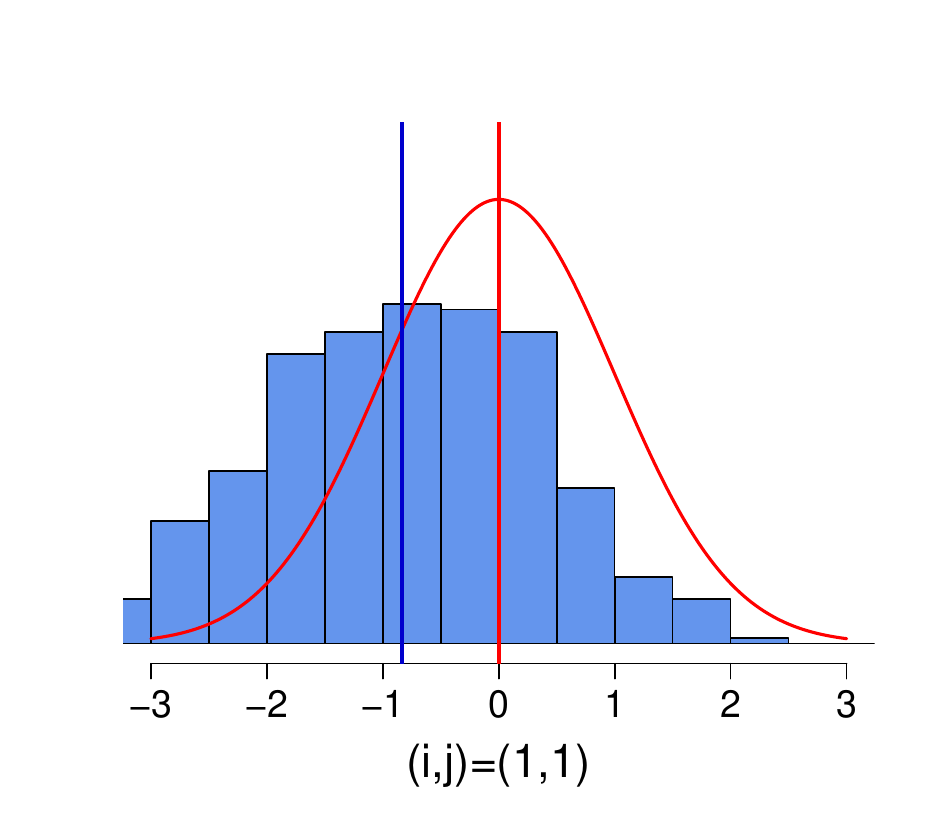}
    \end{minipage}
    \begin{minipage}{0.24\linewidth}
        \centering
        \includegraphics[width=\textwidth]{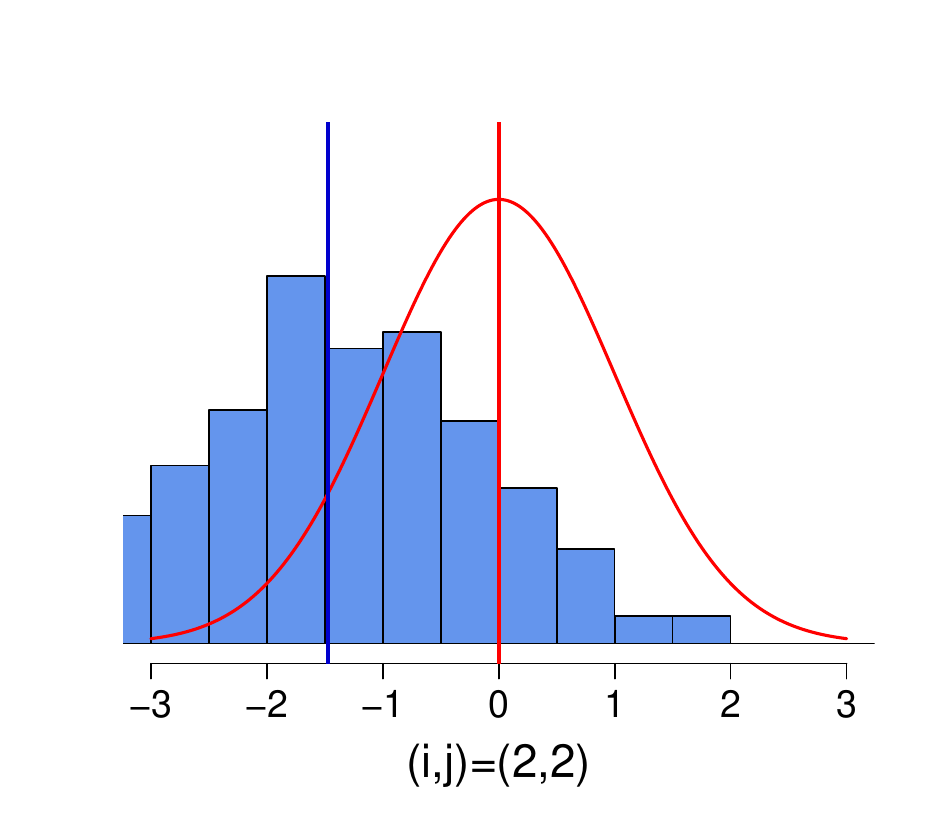}
    \end{minipage}
    \begin{minipage}{0.24\linewidth}
        \centering
        \includegraphics[width=\textwidth]{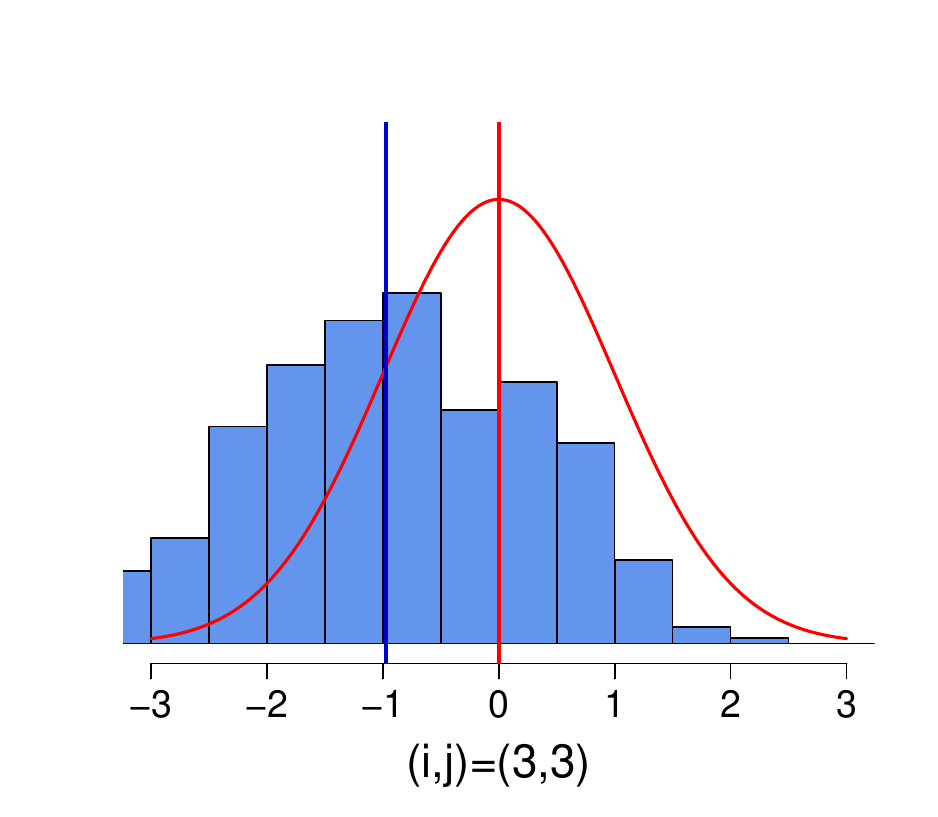}
    \end{minipage}
    \begin{minipage}{0.24\linewidth}
        \centering
        \includegraphics[width=\textwidth]{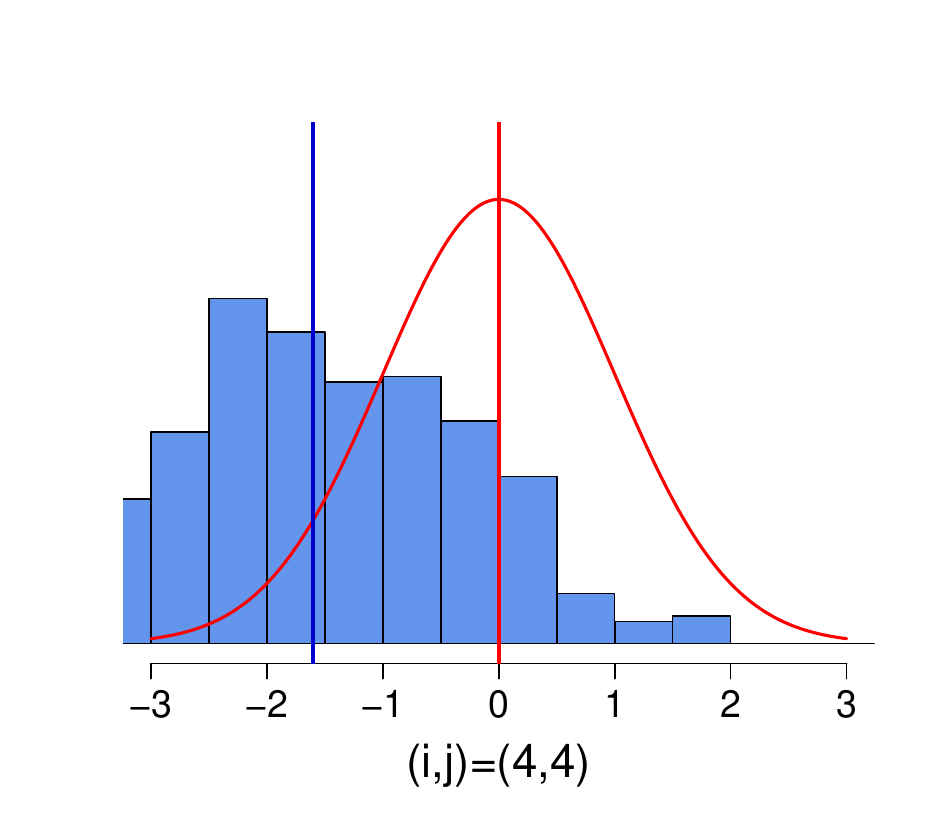}
    \end{minipage}
 \end{minipage}

  \caption*{$n=800, p=400$}
      \vspace{-0.43cm}
 \begin{minipage}{0.3\linewidth}
    \begin{minipage}{0.24\linewidth}
        \centering
        \includegraphics[width=\textwidth]{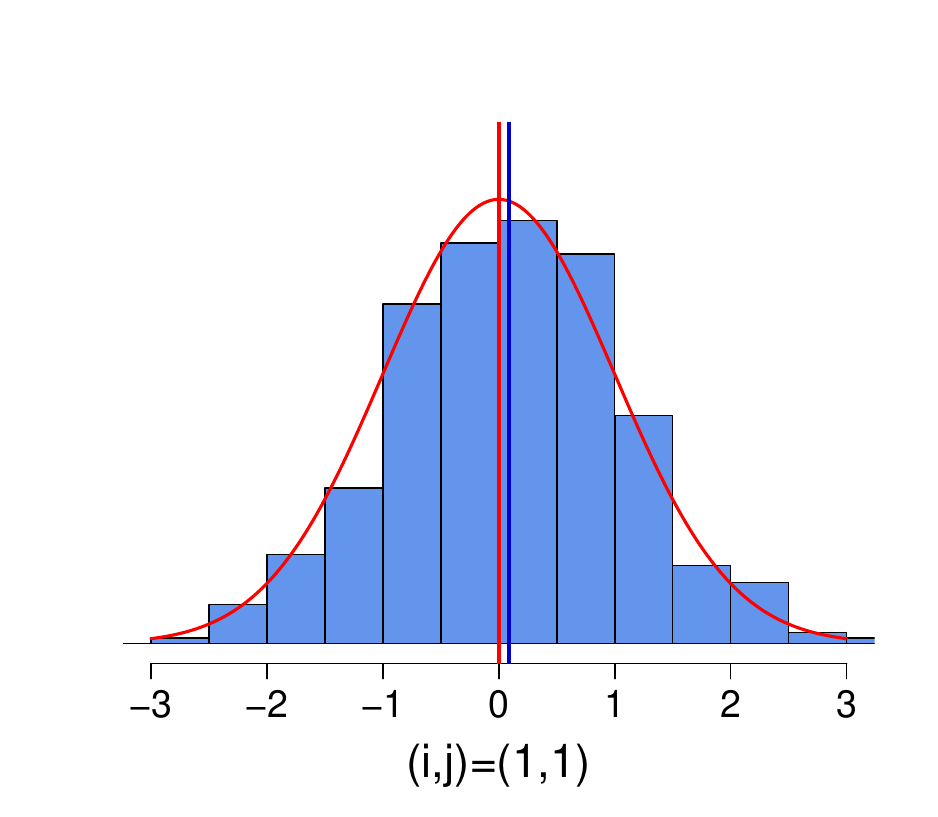}
    \end{minipage}
    \begin{minipage}{0.24\linewidth}
        \centering
        \includegraphics[width=\textwidth]{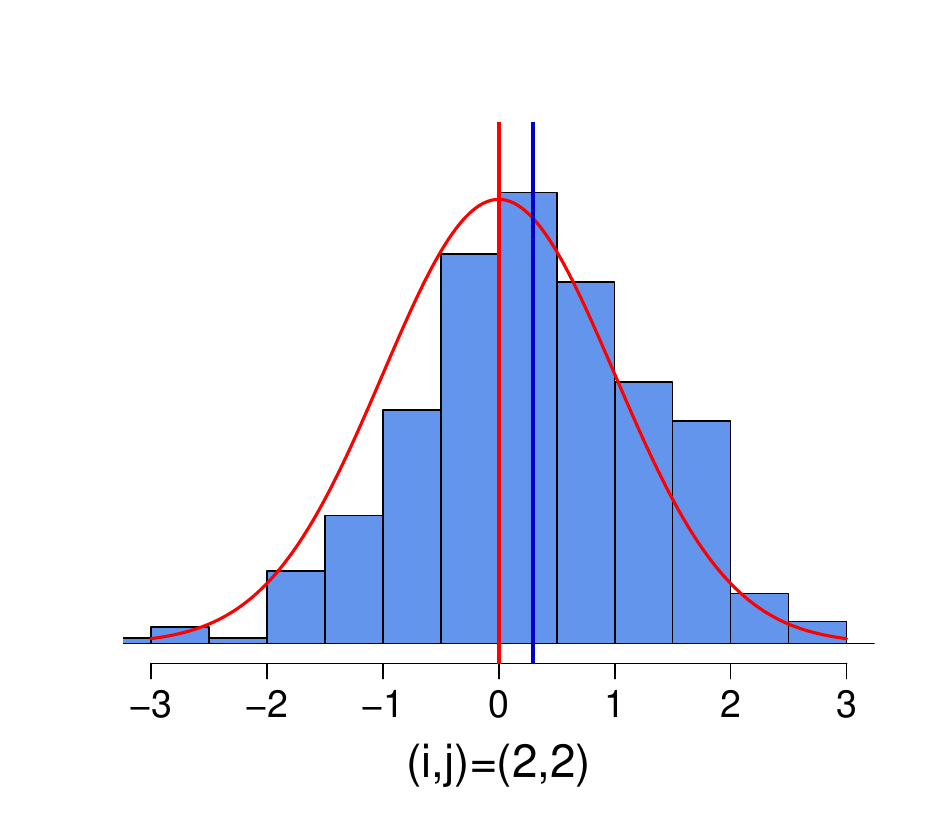}
    \end{minipage}
    \begin{minipage}{0.24\linewidth}
        \centering
        \includegraphics[width=\textwidth]{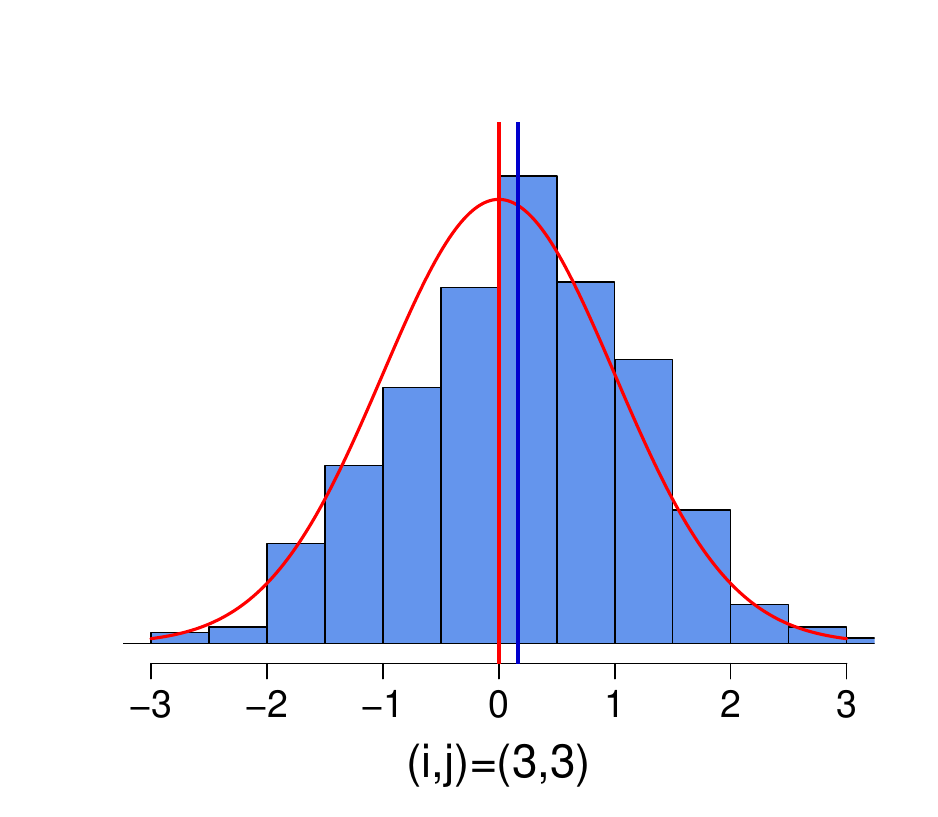}
    \end{minipage}
    \begin{minipage}{0.24\linewidth}
        \centering
        \includegraphics[width=\textwidth]{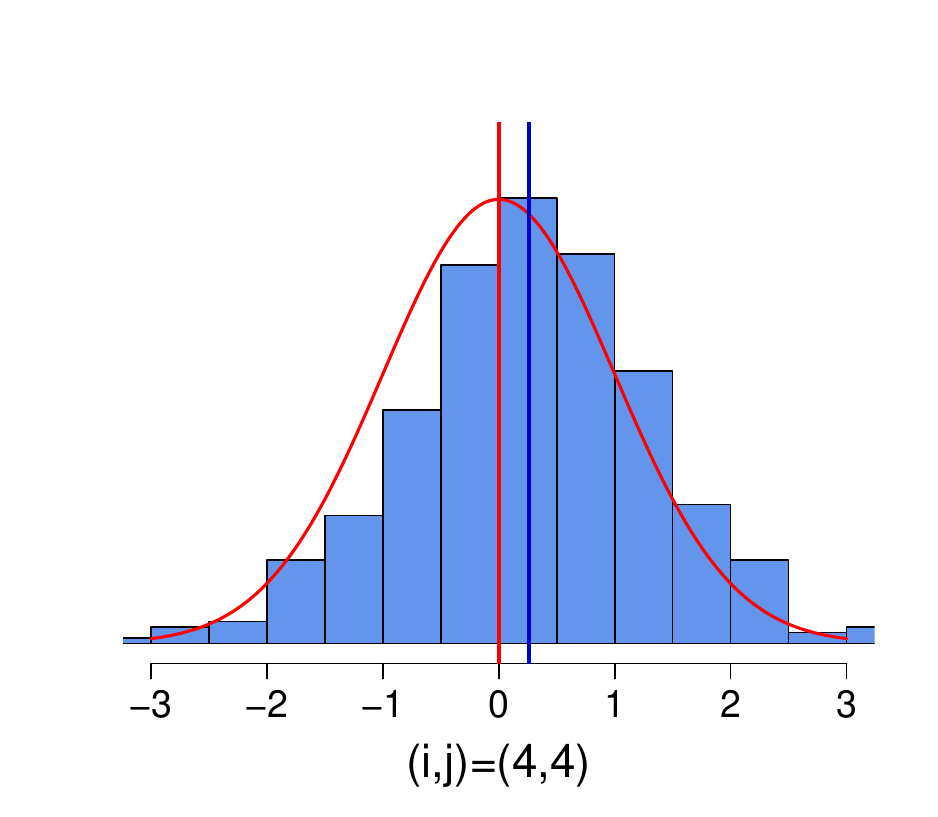}
    \end{minipage}
   \caption*{(a)~~$L_0{:}~ \widehat{\mb{\Omega}}^{\text{US}}$}
 \end{minipage} 
     \hspace{1cm}
 \begin{minipage}{0.3\linewidth}
    \begin{minipage}{0.24\linewidth}
        \centering
        \includegraphics[width=\textwidth]{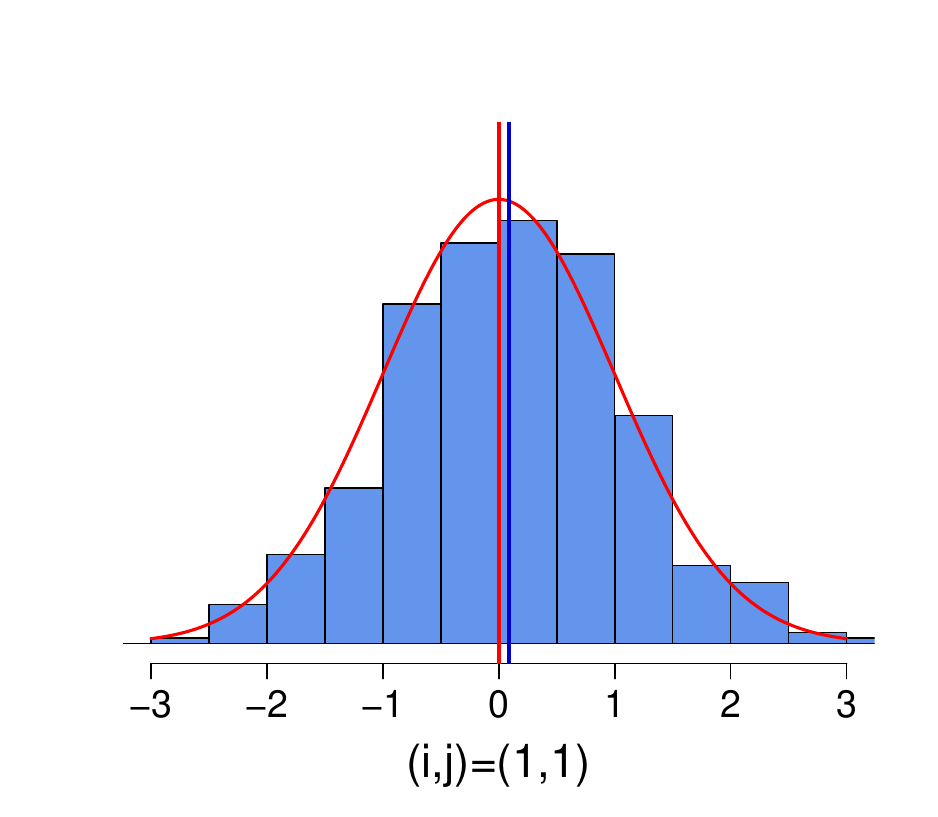}
    \end{minipage}
    \begin{minipage}{0.24\linewidth}
        \centering
        \includegraphics[width=\textwidth]{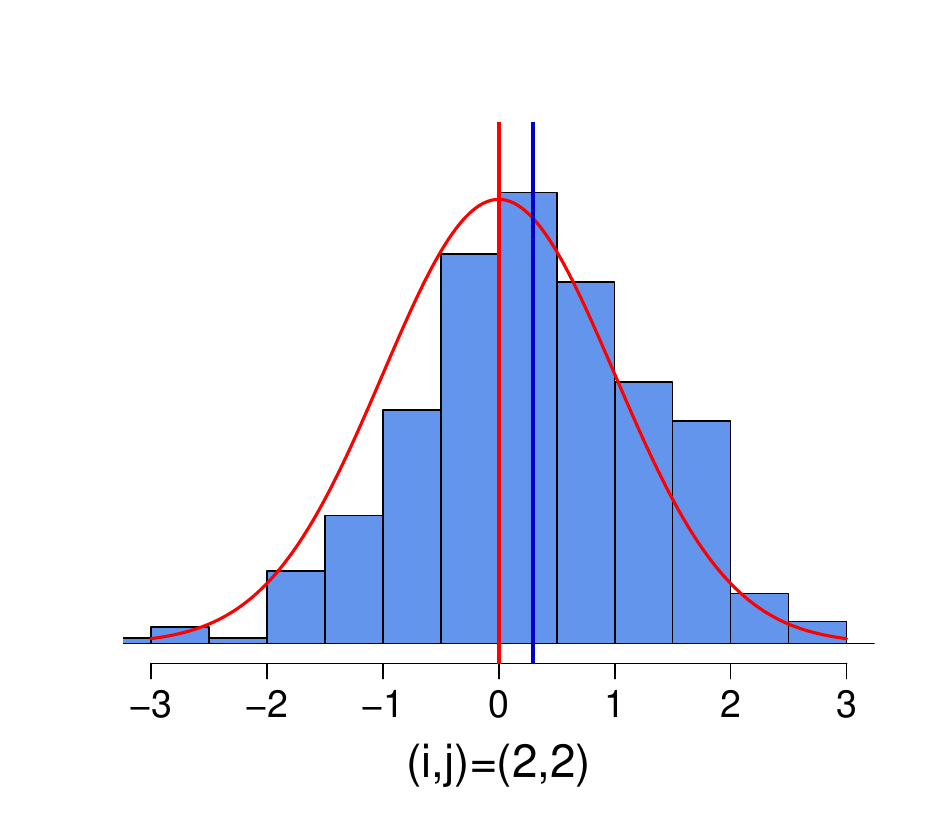}
    \end{minipage}
    \begin{minipage}{0.24\linewidth}
        \centering
        \includegraphics[width=\textwidth]{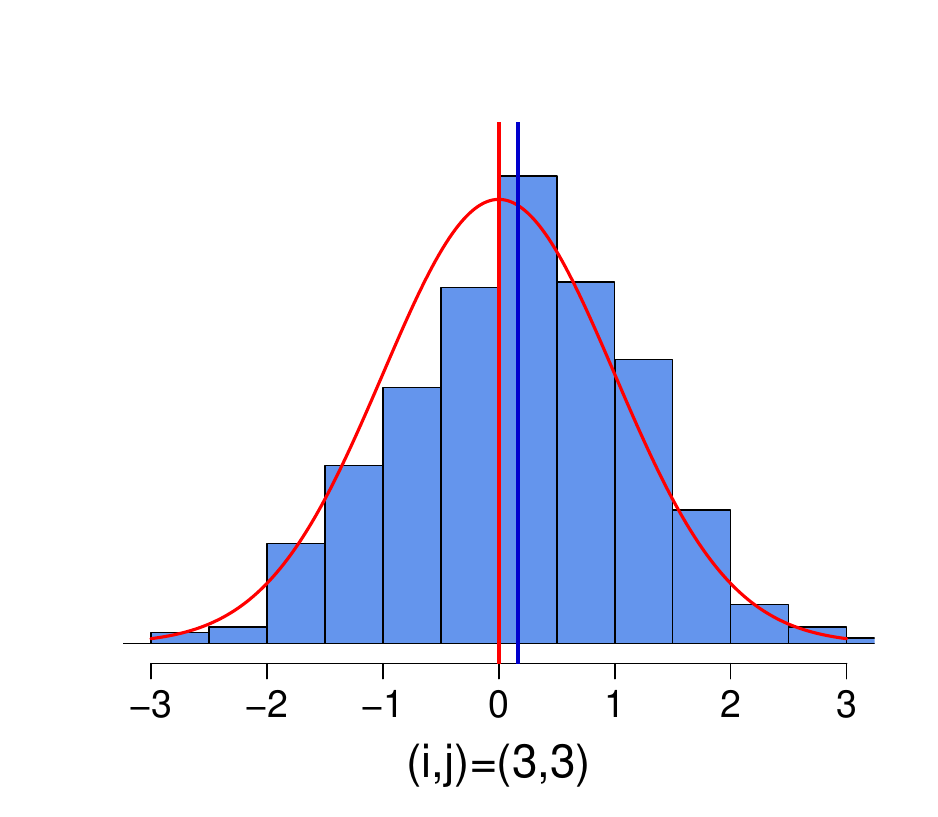}
    \end{minipage}
    \begin{minipage}{0.24\linewidth}
        \centering
        \includegraphics[width=\textwidth]{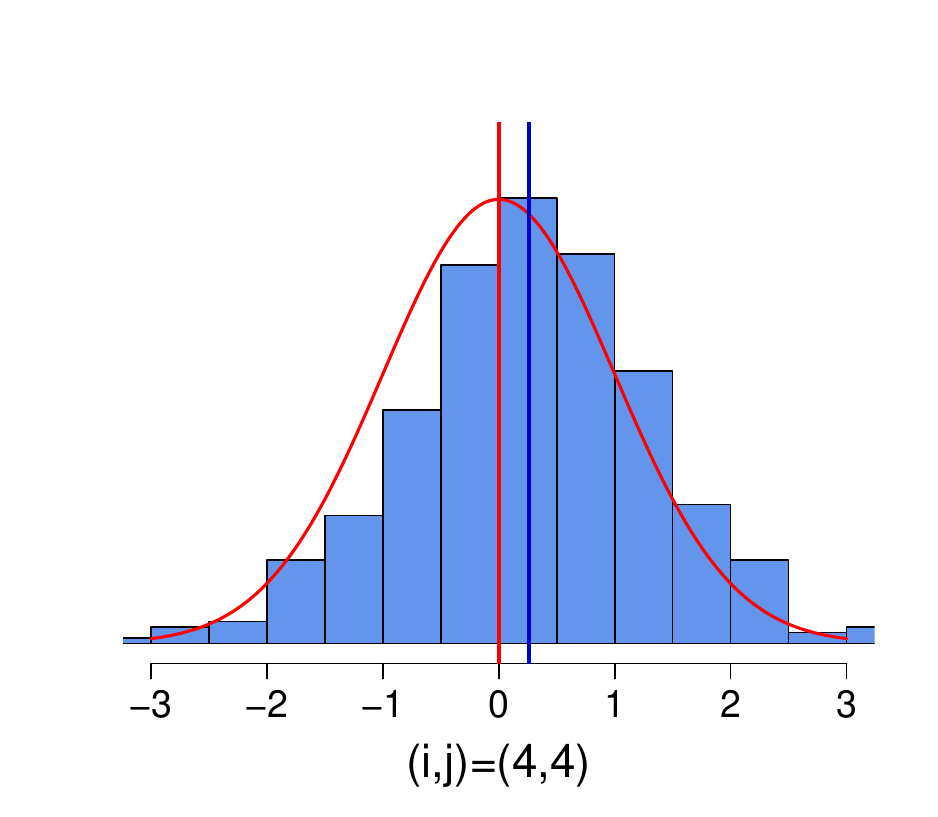}
    \end{minipage}
        \caption*{(b)~~$L_0{:}~ \widehat{\mb{T}}$}
 \end{minipage}   
      \hspace{1cm}
 \begin{minipage}{0.3\linewidth}
    \begin{minipage}{0.24\linewidth}
        \centering
        \includegraphics[width=\textwidth]{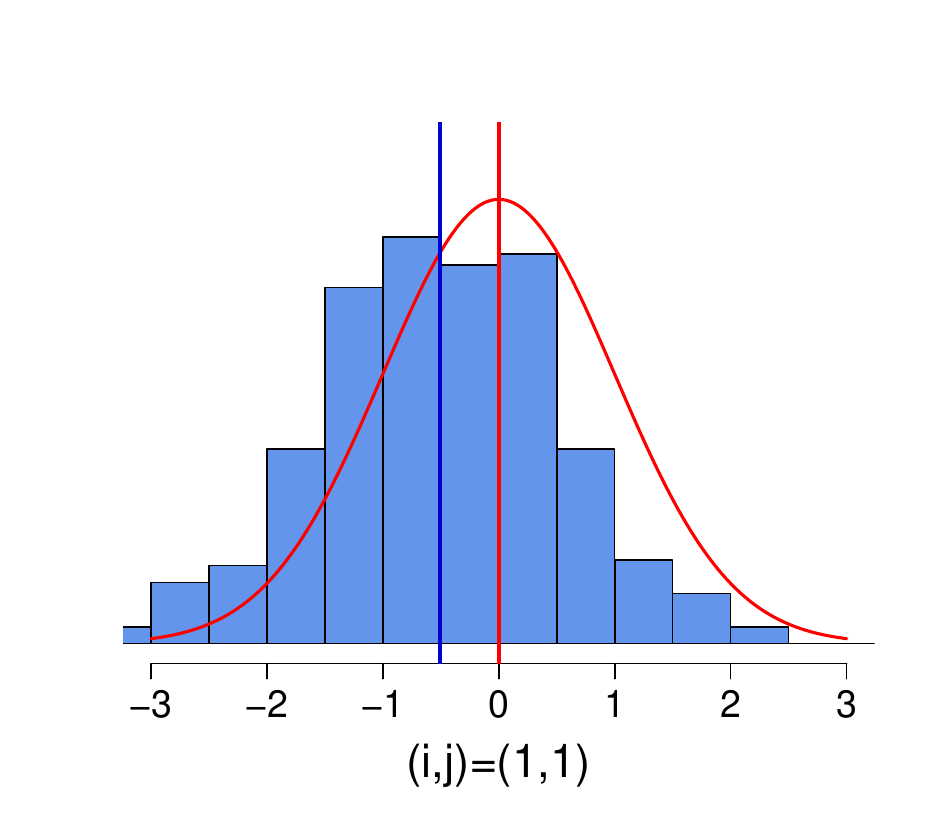}
    \end{minipage}
    \begin{minipage}{0.24\linewidth}
        \centering
        \includegraphics[width=\textwidth]{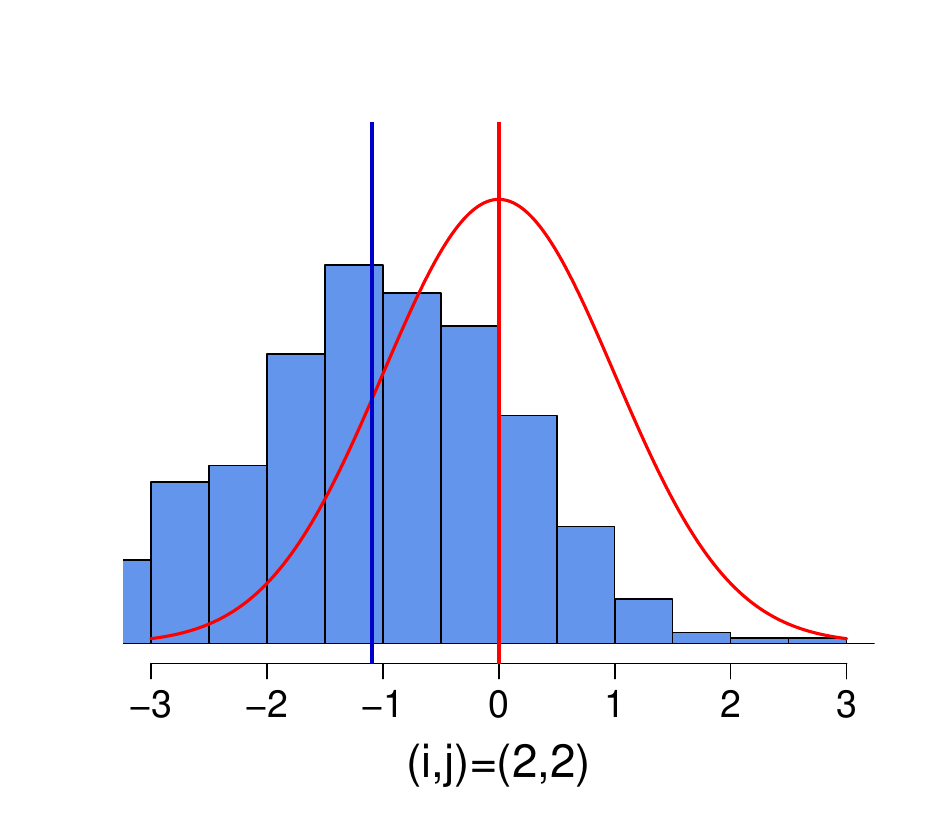}
    \end{minipage}
    \begin{minipage}{0.24\linewidth}
        \centering
        \includegraphics[width=\textwidth]{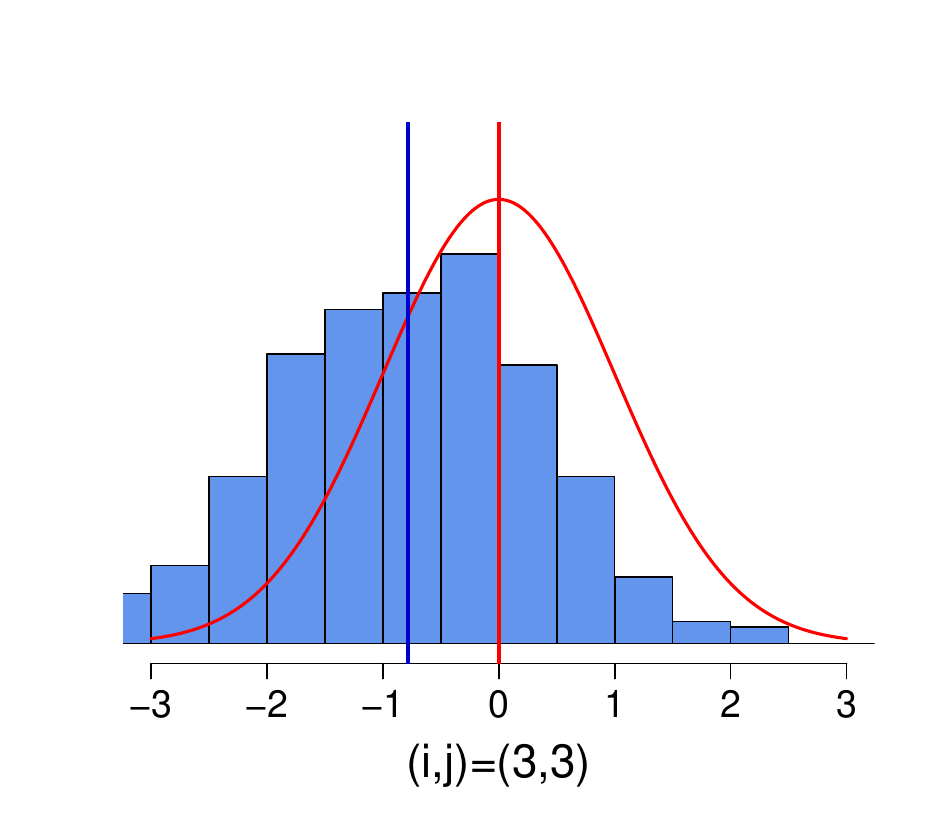}
    \end{minipage}
    \begin{minipage}{0.24\linewidth}
        \centering
        \includegraphics[width=\textwidth]{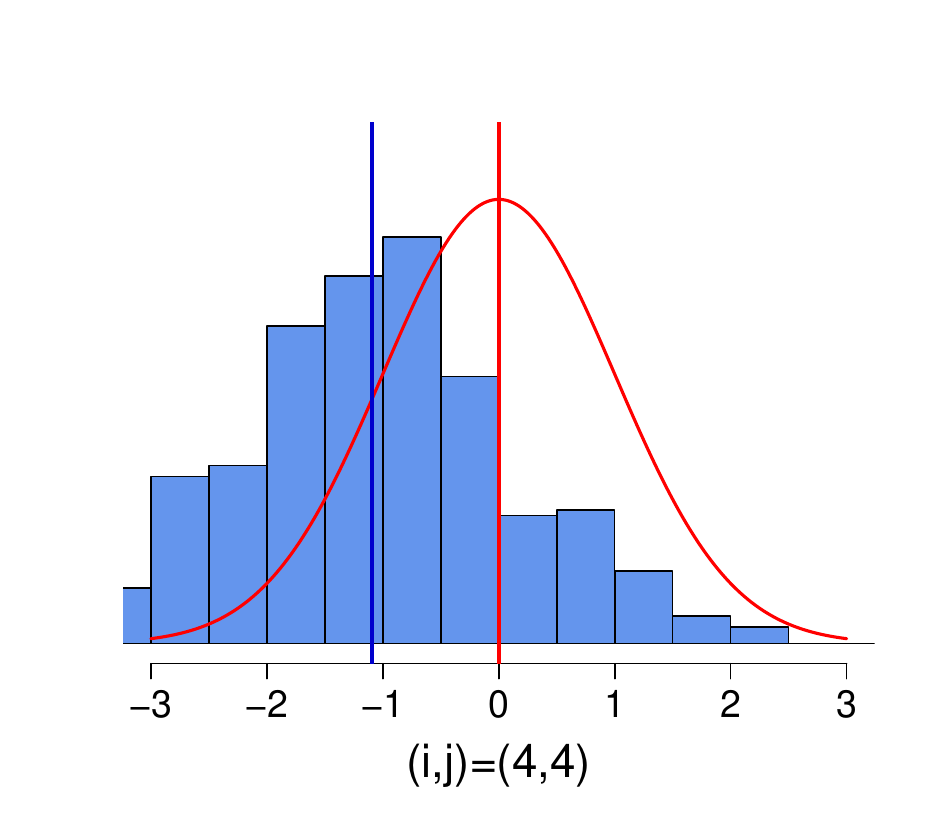}
    \end{minipage}
    \caption*{(c)~~$L_1{:}~ \widehat{\mb{T}}$}
     \end{minipage}   
     \caption{Histograms of $\big(\sqrt{n}(\widehat{\mb{\Omega}}_{ij}^{(m)}-\mb{\Omega}_{ij})/\widehat{\sigma}_{\mb{\Omega}_{ij}}^{(m)}\big)_{m=1}^{400}$ under Gaussian random graph settings.}
\end{sidewaysfigure}

 %Gaussian hub
 \begin{sidewaysfigure}[th!]
  \caption*{$n=200, p=200$}
      \vspace{-0.43cm}
 \begin{minipage}{0.3\linewidth}
    \begin{minipage}{0.24\linewidth}
        \centering
        \includegraphics[width=\textwidth]{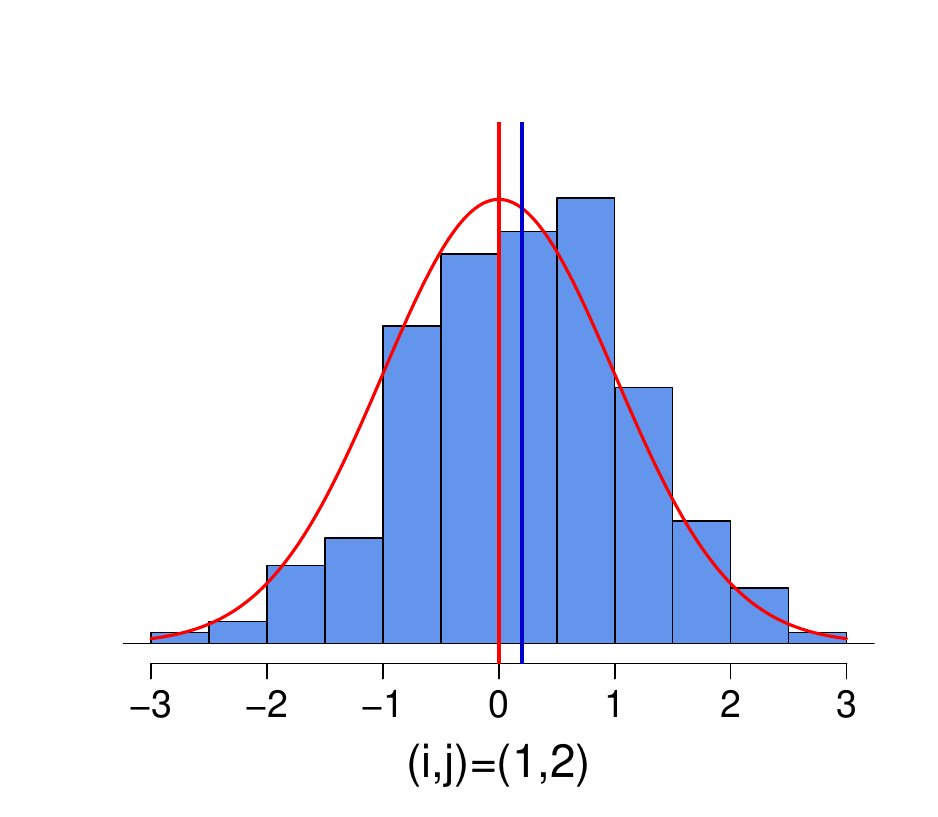}
    \end{minipage}
    \begin{minipage}{0.24\linewidth}
        \centering
        \includegraphics[width=\textwidth]{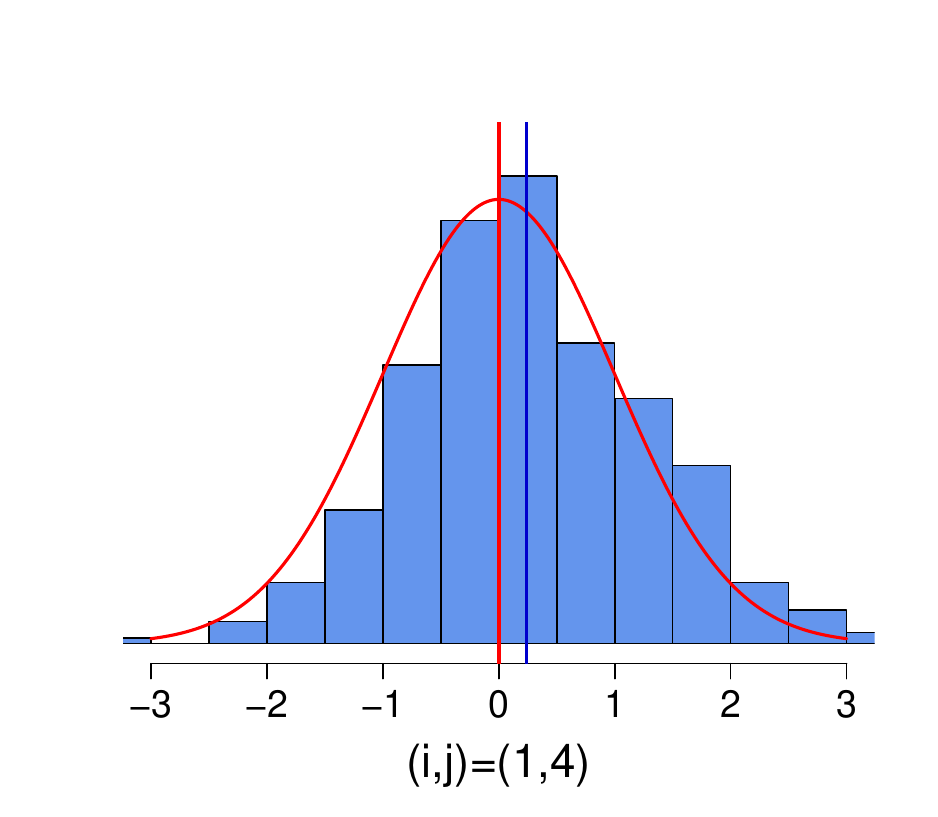}
    \end{minipage}
    \begin{minipage}{0.24\linewidth}
        \centering
        \includegraphics[width=\textwidth]{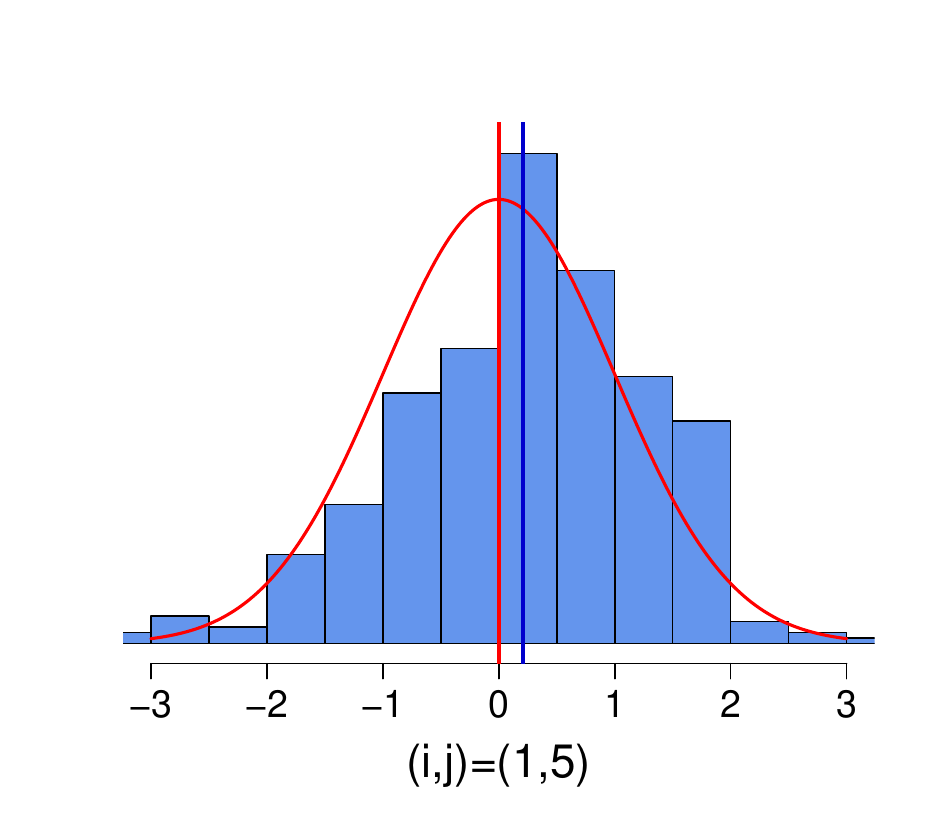}
    \end{minipage}
    \begin{minipage}{0.24\linewidth}
        \centering
        \includegraphics[width=\textwidth]{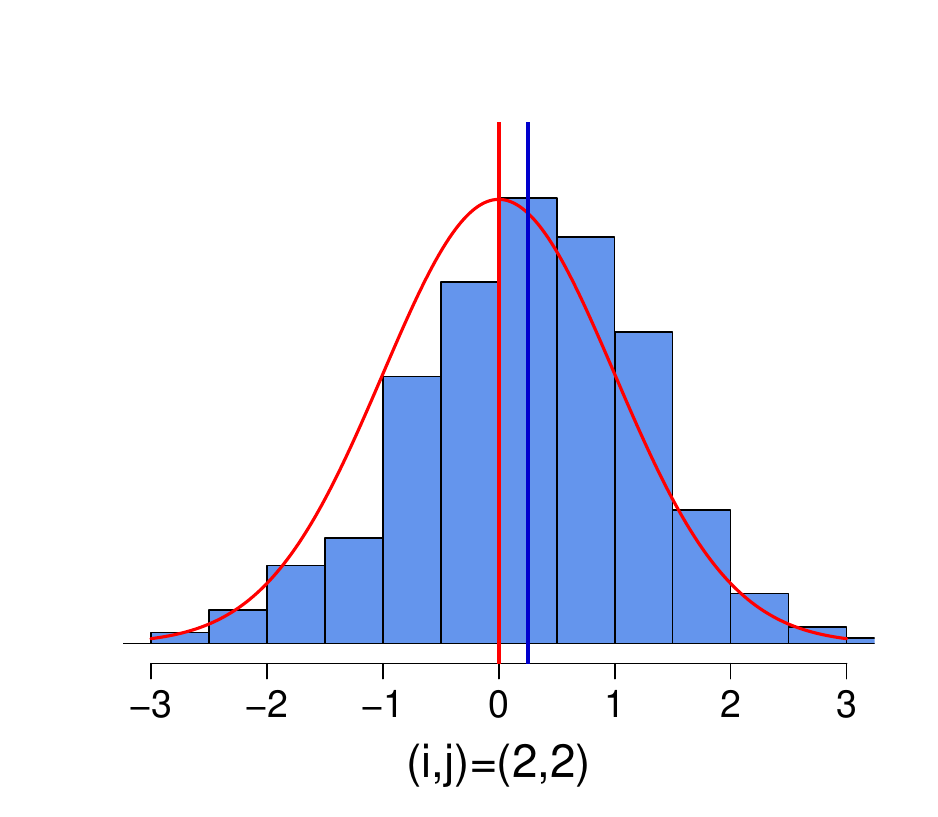}
    \end{minipage}
 \end{minipage}
 \hspace{1cm}
 \begin{minipage}{0.3\linewidth}
    \begin{minipage}{0.24\linewidth}
        \centering
        \includegraphics[width=\textwidth]{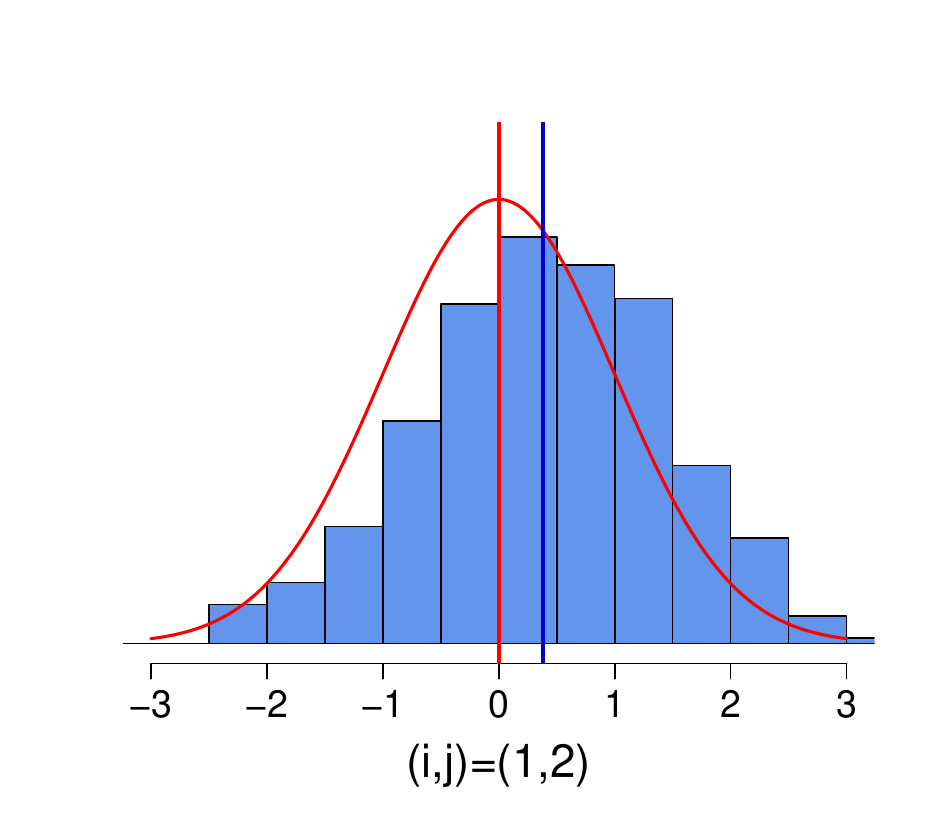}
    \end{minipage}
    \begin{minipage}{0.24\linewidth}
        \centering
        \includegraphics[width=\textwidth]{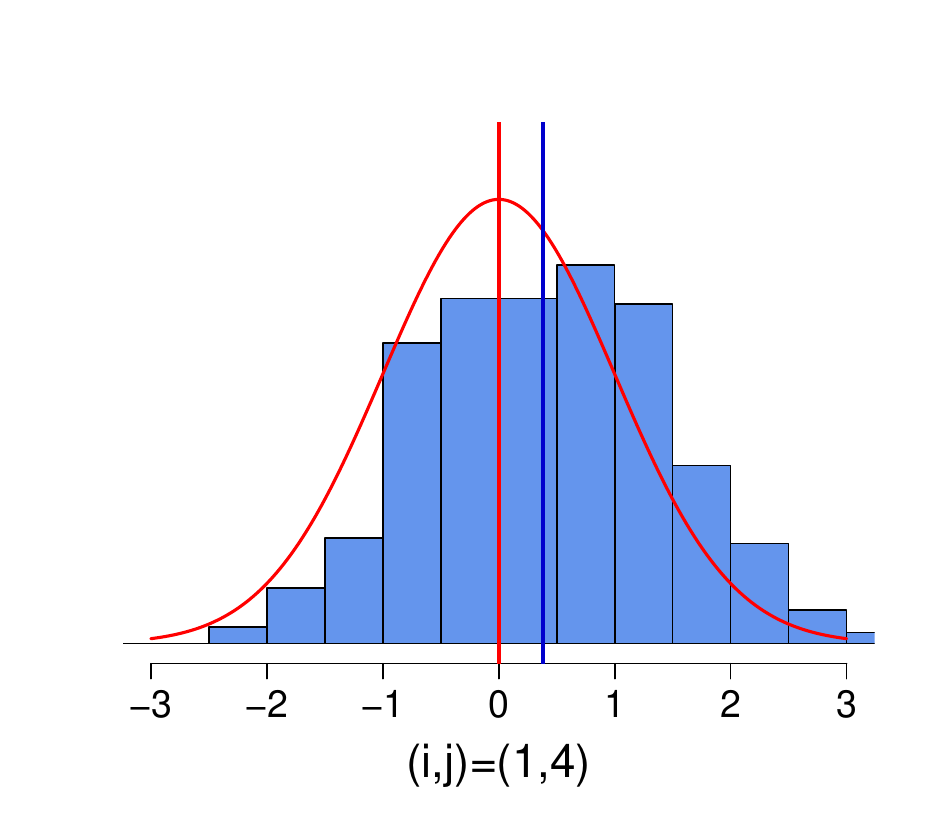}
    \end{minipage}
    \begin{minipage}{0.24\linewidth}
        \centering
        \includegraphics[width=\textwidth]{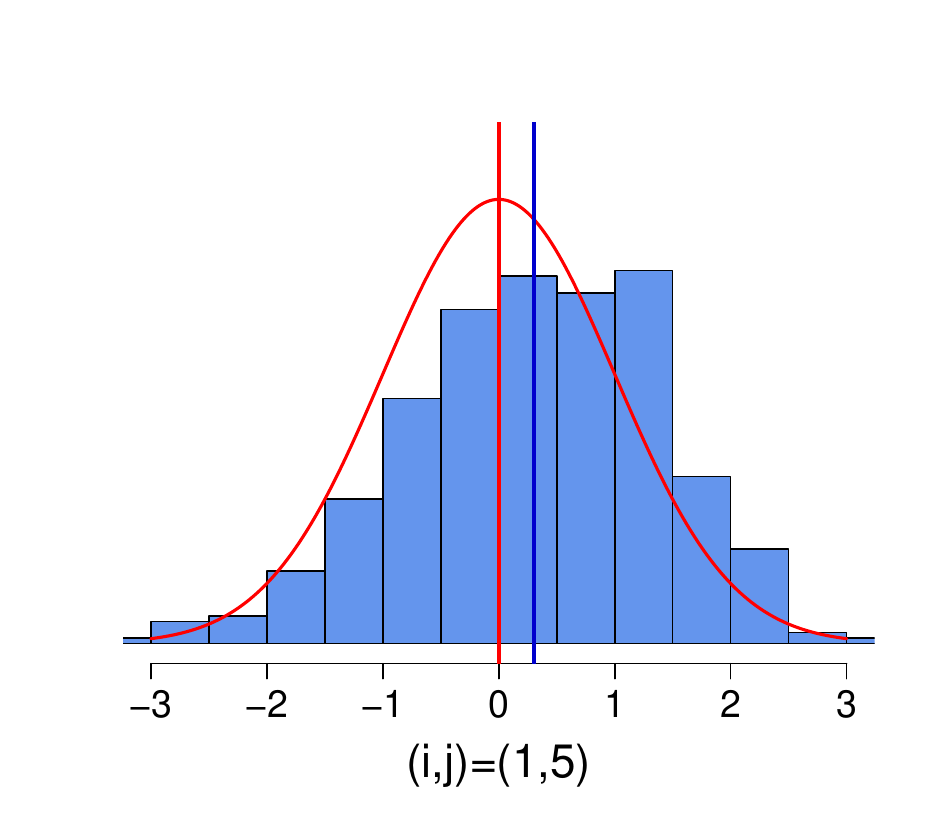}
    \end{minipage}
    \begin{minipage}{0.24\linewidth}
        \centering
        \includegraphics[width=\textwidth]{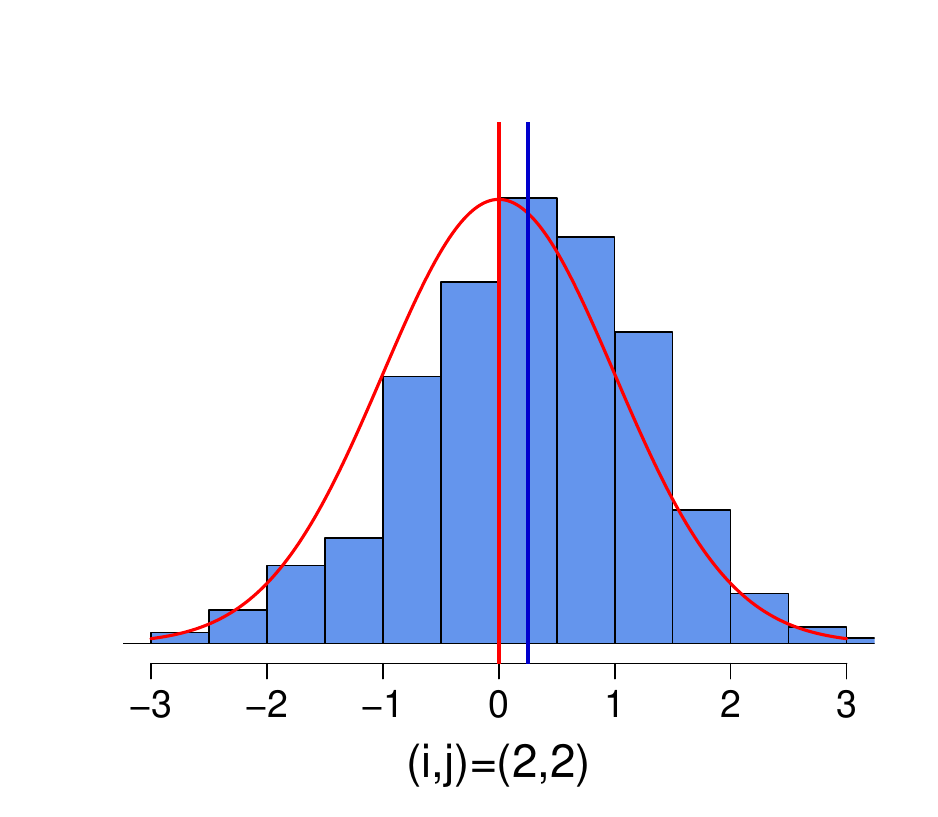}
    \end{minipage}    
 \end{minipage}
  \hspace{1cm}
 \begin{minipage}{0.3\linewidth}
     \begin{minipage}{0.24\linewidth}
        \centering
        \includegraphics[width=\textwidth]{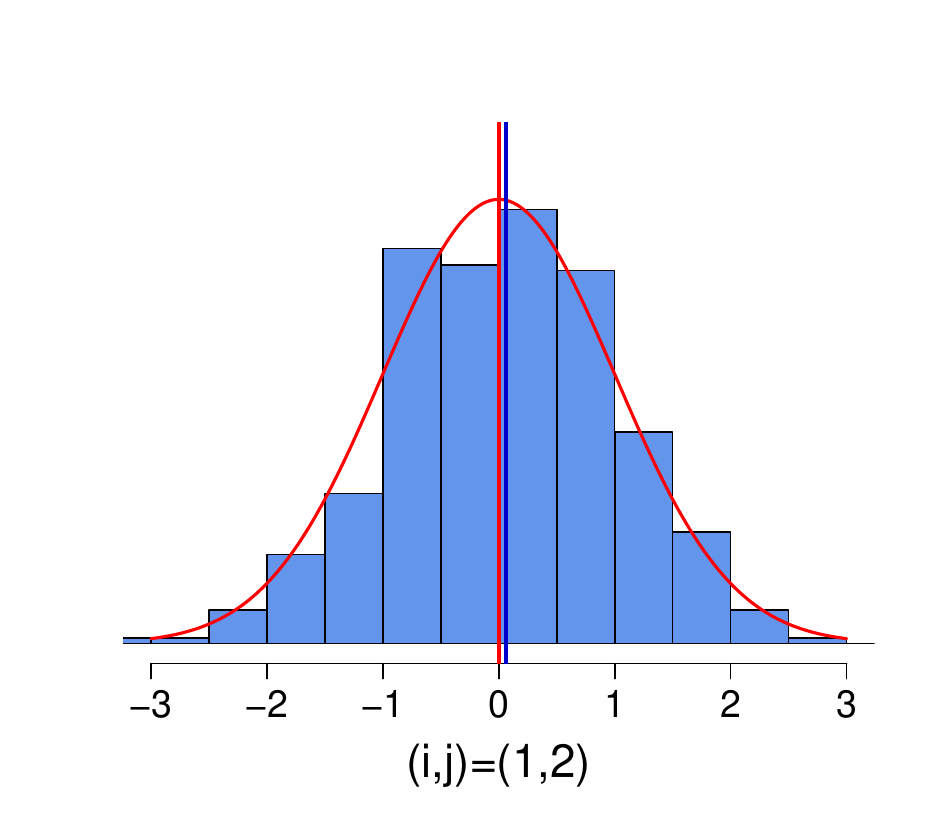}
    \end{minipage}
    \begin{minipage}{0.24\linewidth}
        \centering
        \includegraphics[width=\textwidth]{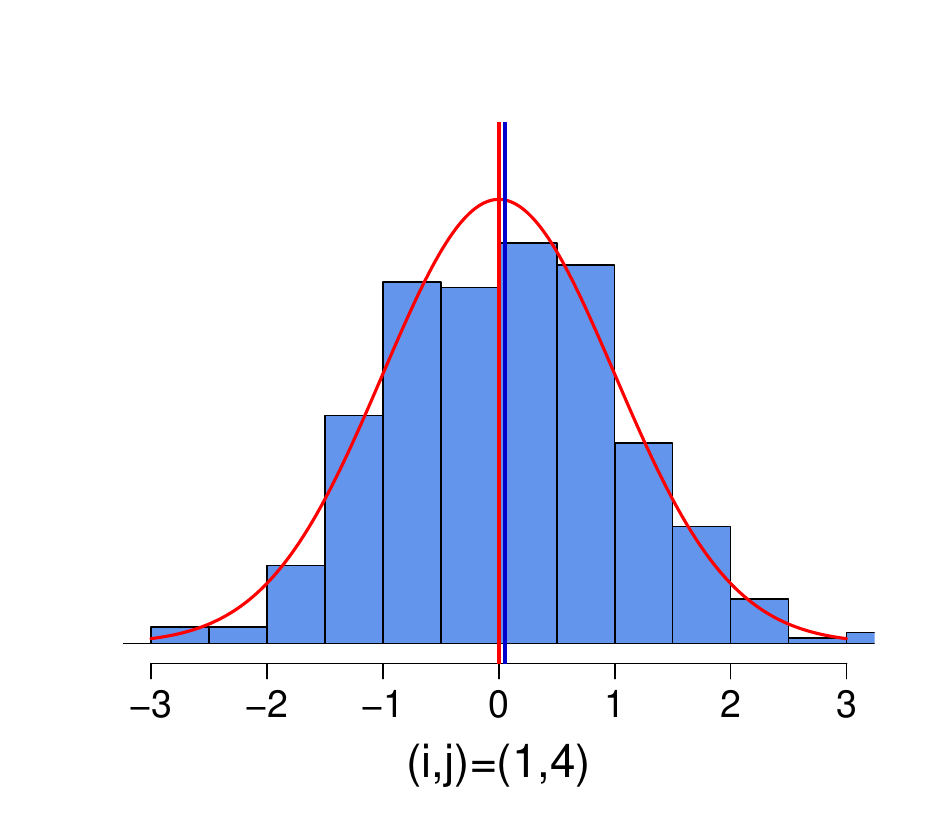}
    \end{minipage}
    \begin{minipage}{0.24\linewidth}
        \centering
        \includegraphics[width=\textwidth]{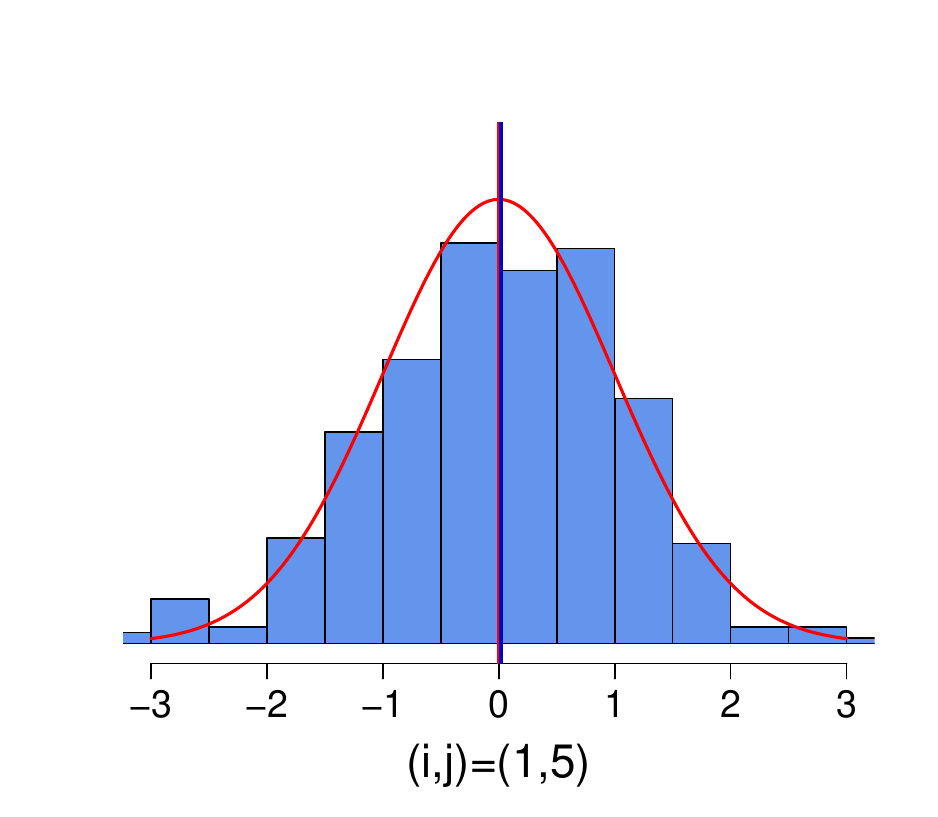}
    \end{minipage}
    \begin{minipage}{0.24\linewidth}
        \centering
        \includegraphics[width=\textwidth]{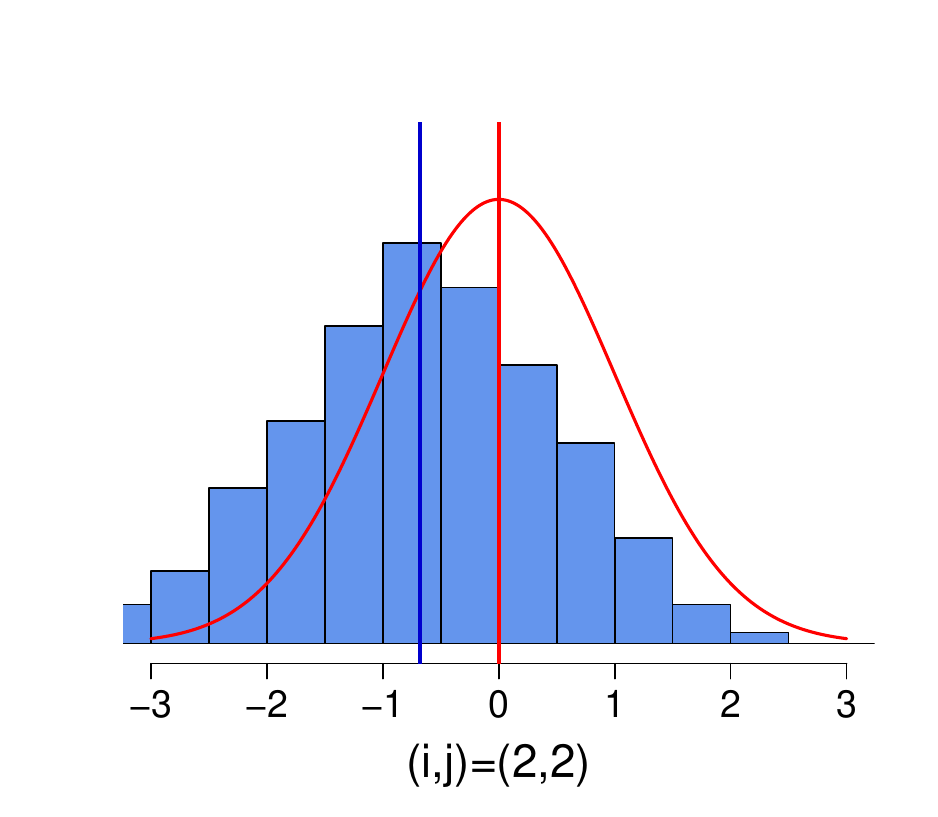}
    \end{minipage}
 \end{minipage}

  \caption*{$n=400, p=200$}
      \vspace{-0.43cm}
 \begin{minipage}{0.3\linewidth}
    \begin{minipage}{0.24\linewidth}
        \centering
        \includegraphics[width=\textwidth]{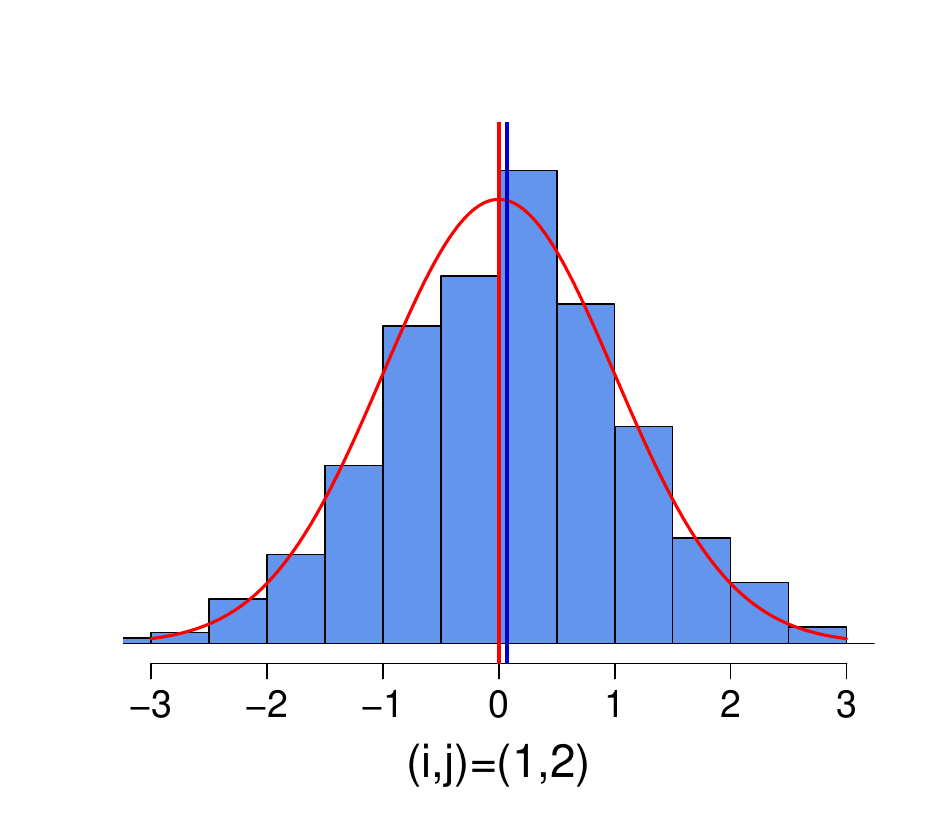}
    \end{minipage}
    \begin{minipage}{0.24\linewidth}
        \centering
        \includegraphics[width=\textwidth]{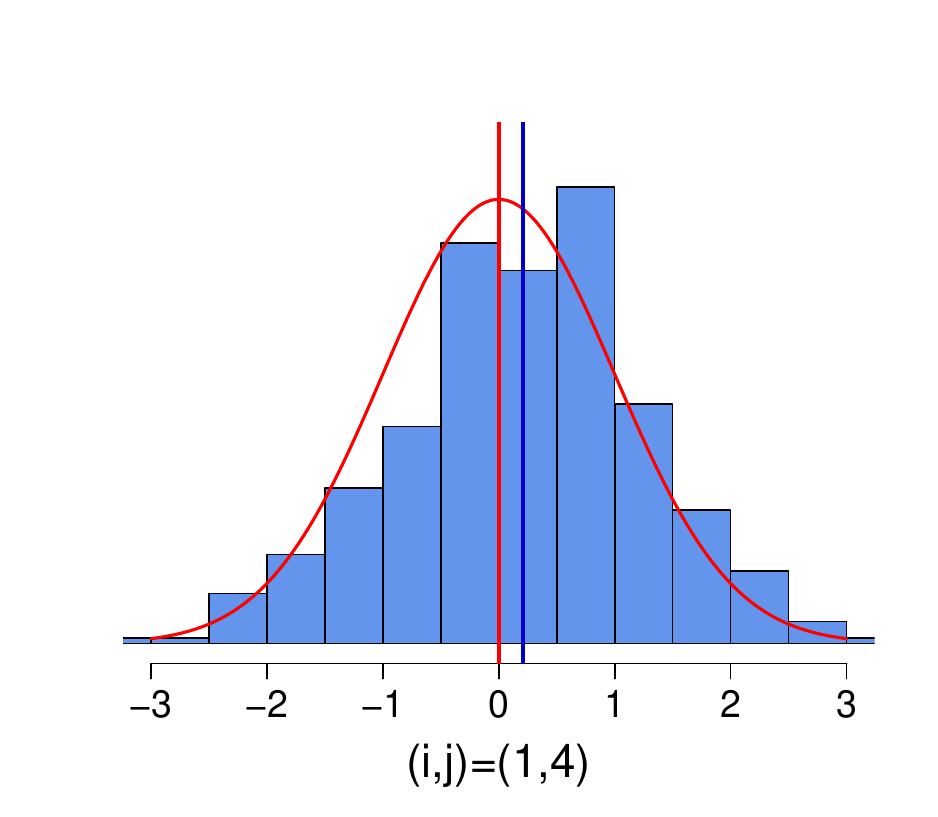}
    \end{minipage}
    \begin{minipage}{0.24\linewidth}
        \centering
        \includegraphics[width=\textwidth]{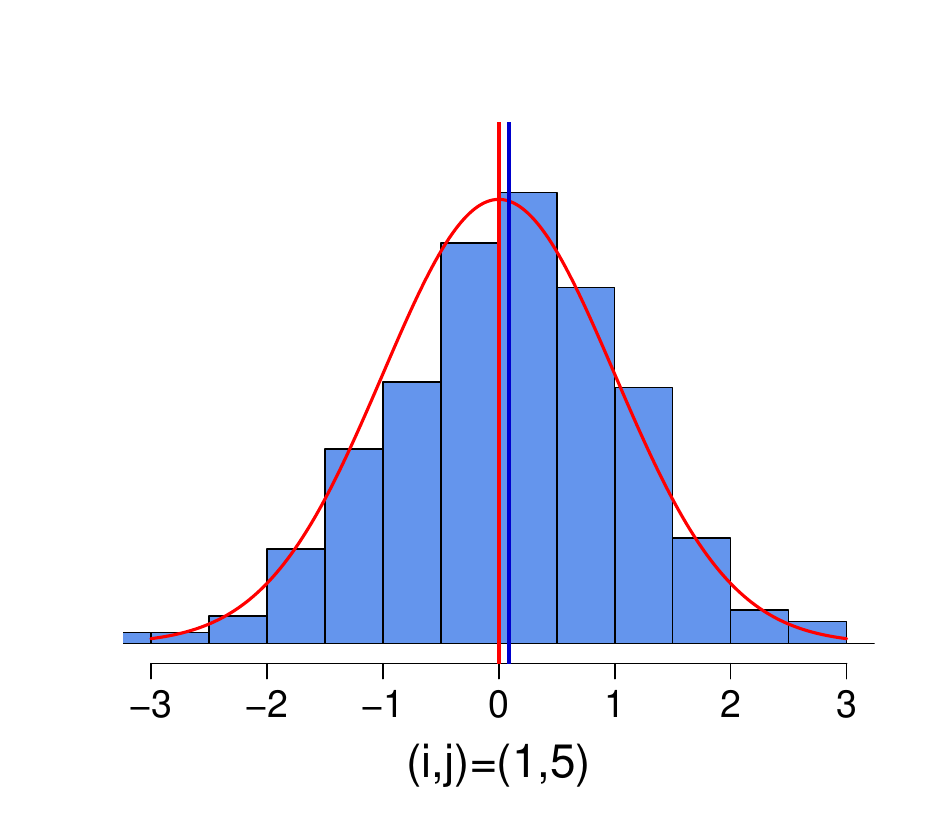}
    \end{minipage}
    \begin{minipage}{0.24\linewidth}
        \centering
        \includegraphics[width=\textwidth]{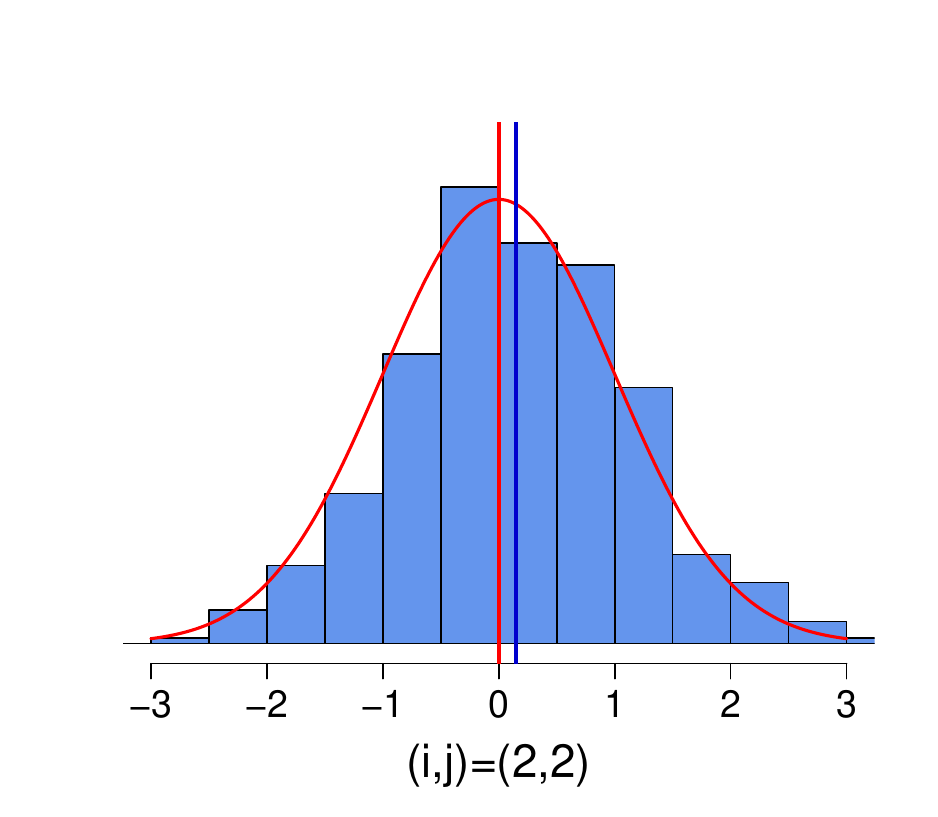}
    \end{minipage}
 \end{minipage}  
     \hspace{1cm}
 \begin{minipage}{0.3\linewidth}
    \begin{minipage}{0.24\linewidth}
        \centering
        \includegraphics[width=\textwidth]{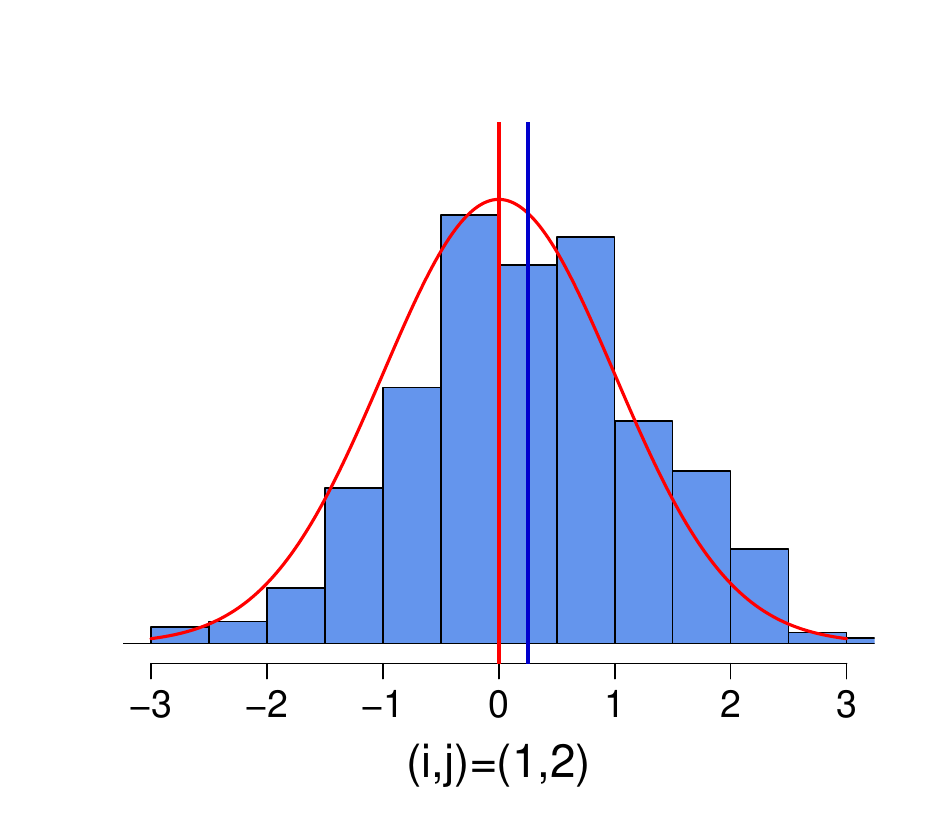}
    \end{minipage}
    \begin{minipage}{0.24\linewidth}
        \centering
        \includegraphics[width=\textwidth]{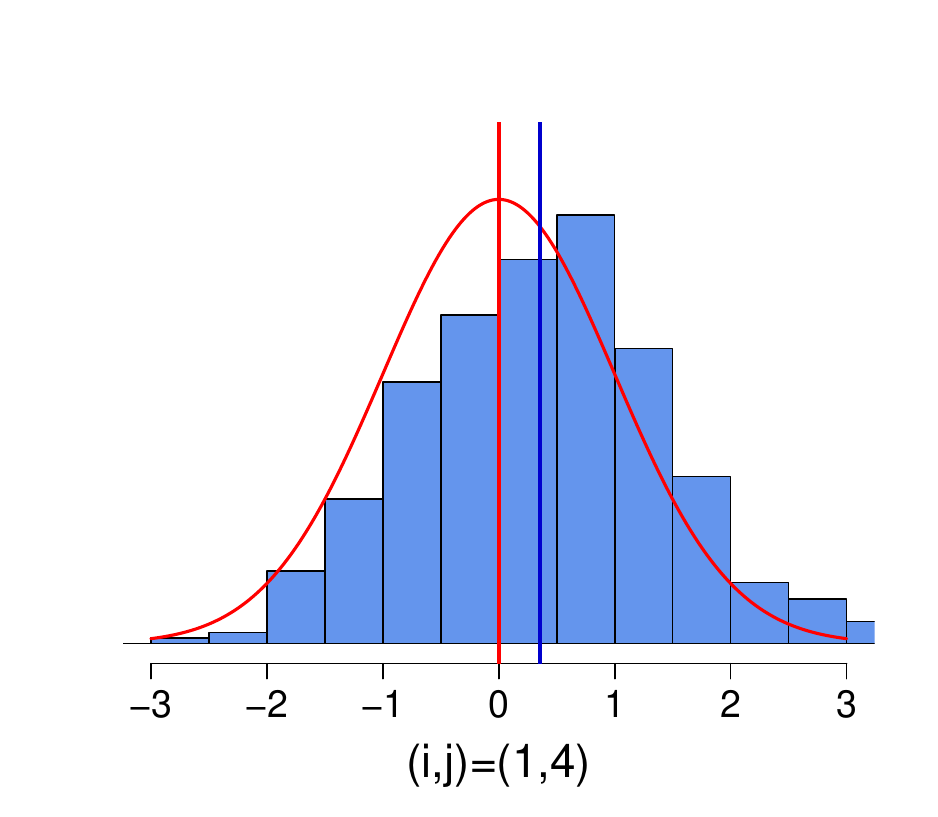}
    \end{minipage}
    \begin{minipage}{0.24\linewidth}
        \centering
        \includegraphics[width=\textwidth]{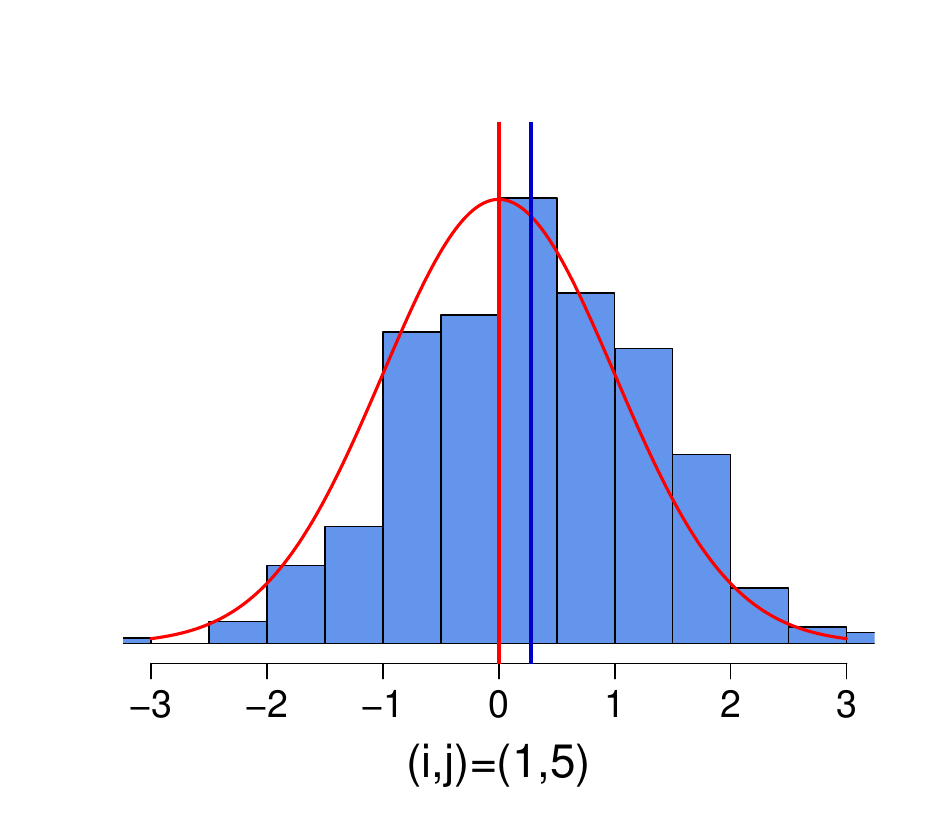}
    \end{minipage}
    \begin{minipage}{0.24\linewidth}
        \centering
        \includegraphics[width=\textwidth]{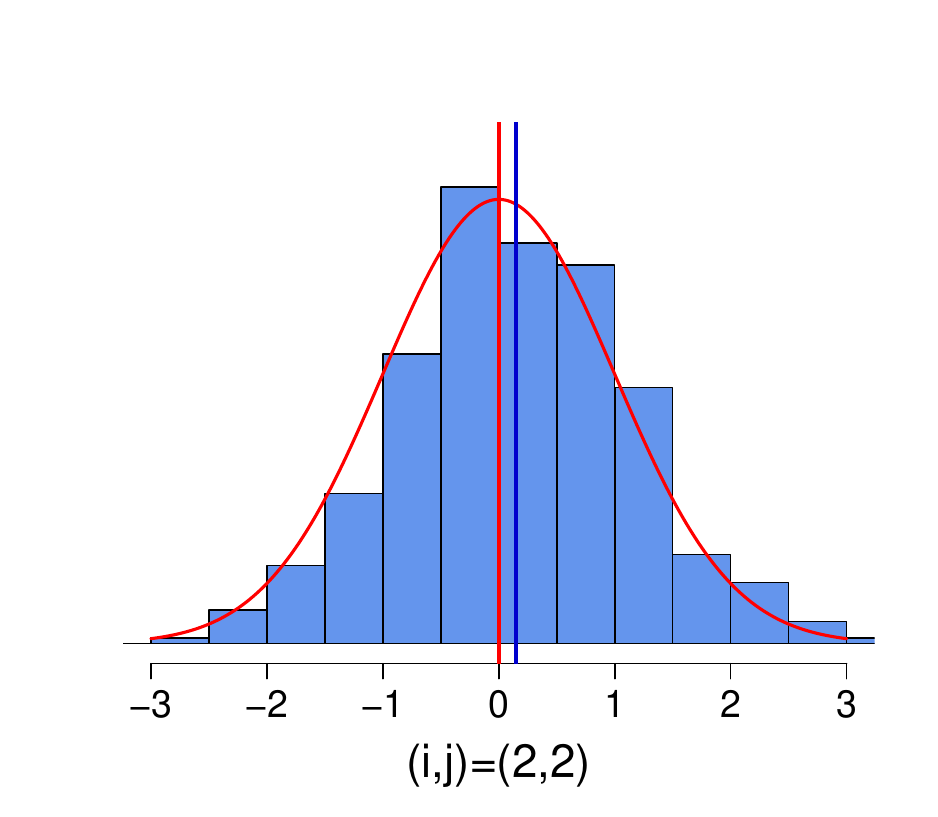}
    \end{minipage}
  \end{minipage}  
    \hspace{1cm}
 \begin{minipage}{0.3\linewidth}
    \begin{minipage}{0.24\linewidth}
        \centering
        \includegraphics[width=\textwidth]{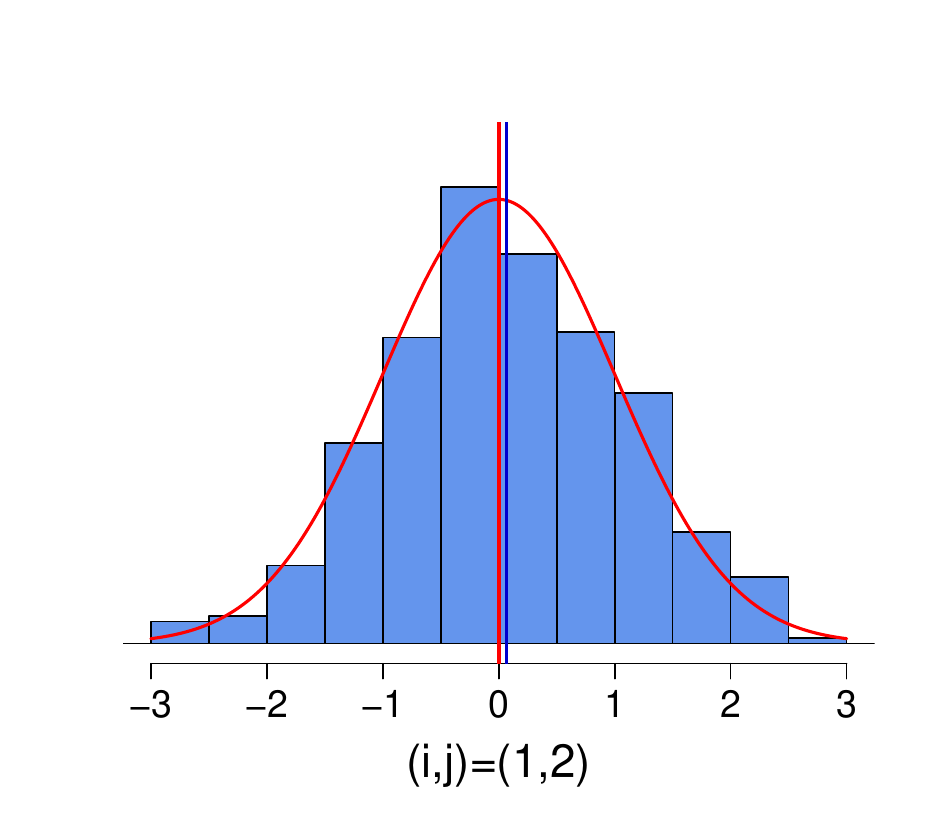}
    \end{minipage}
    \begin{minipage}{0.24\linewidth}
        \centering
        \includegraphics[width=\textwidth]{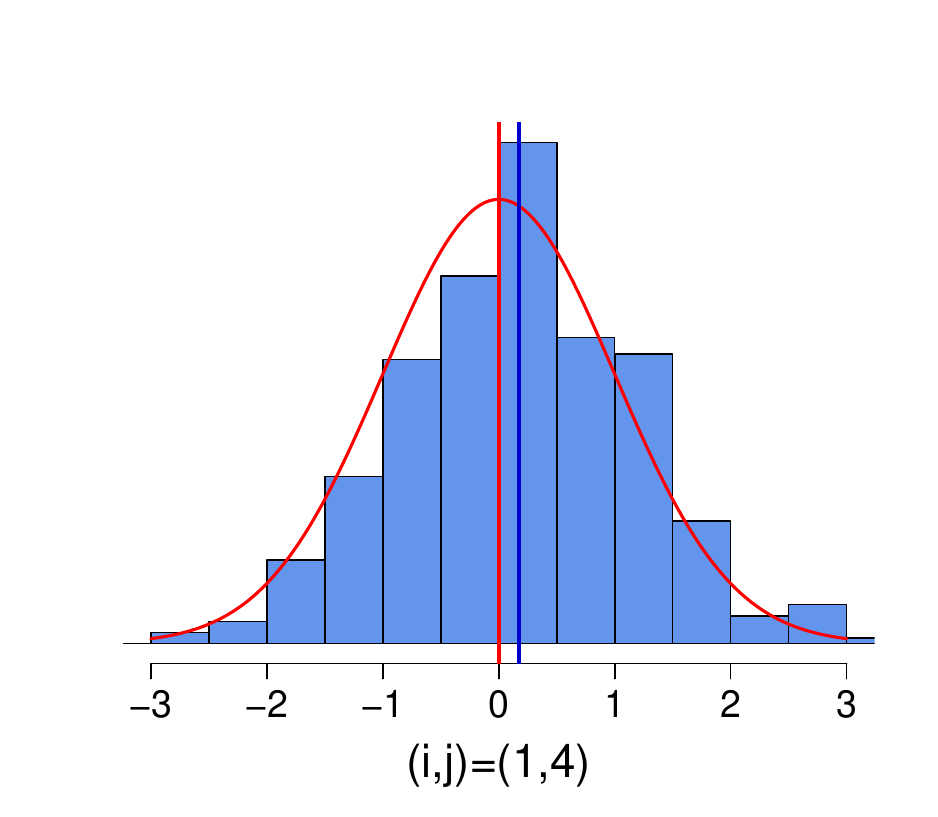}
    \end{minipage}
    \begin{minipage}{0.24\linewidth}
        \centering
        \includegraphics[width=\textwidth]{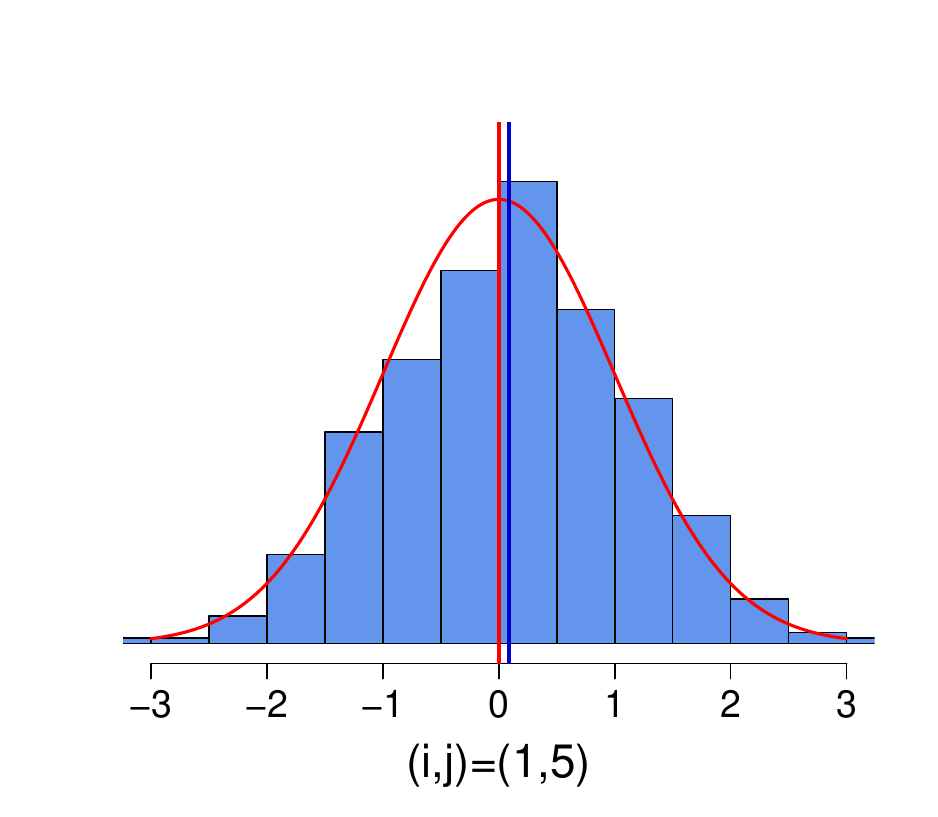}
    \end{minipage}
    \begin{minipage}{0.24\linewidth}
        \centering
        \includegraphics[width=\textwidth]{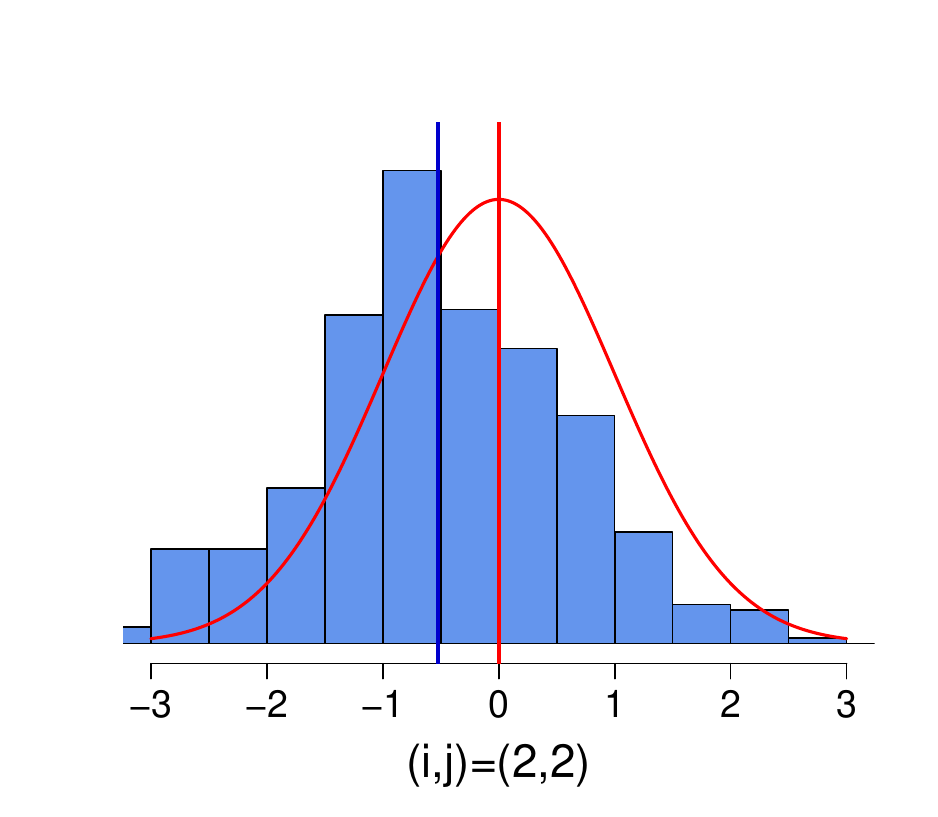}
    \end{minipage}
 \end{minipage}

  \caption*{$n=800, p=200$}
      \vspace{-0.43cm}
 \begin{minipage}{0.3\linewidth}
    \begin{minipage}{0.24\linewidth}
        \centering
        \includegraphics[width=\textwidth]{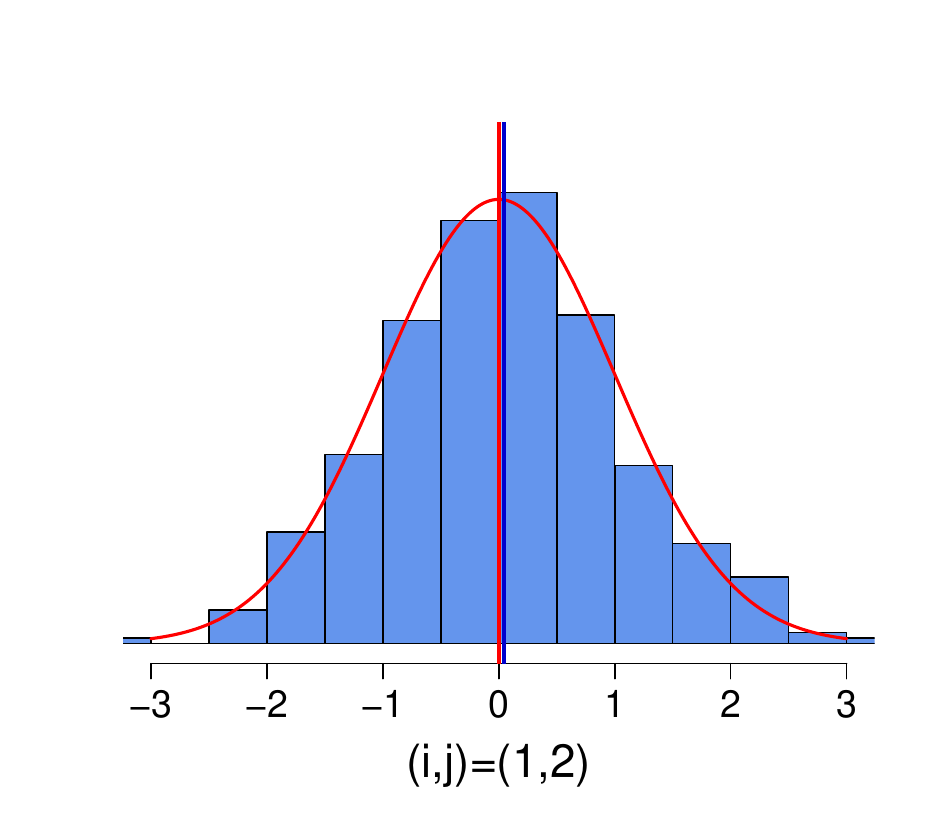}
    \end{minipage}
    \begin{minipage}{0.24\linewidth}
        \centering
        \includegraphics[width=\textwidth]{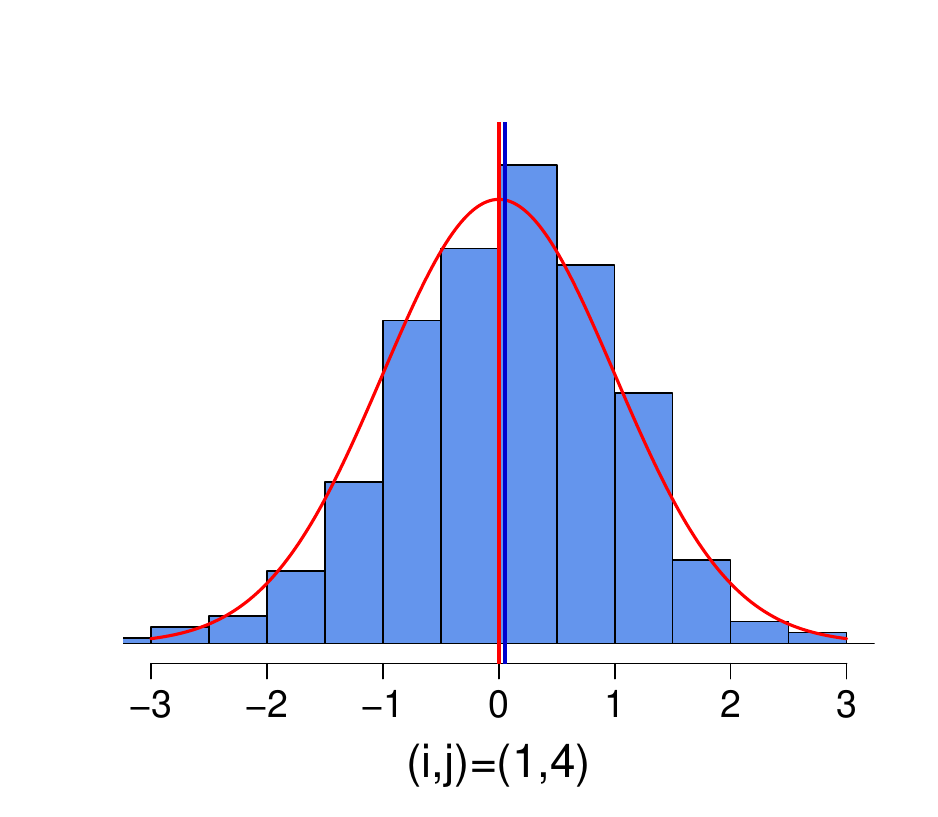}
    \end{minipage}
    \begin{minipage}{0.24\linewidth}
        \centering
        \includegraphics[width=\textwidth]{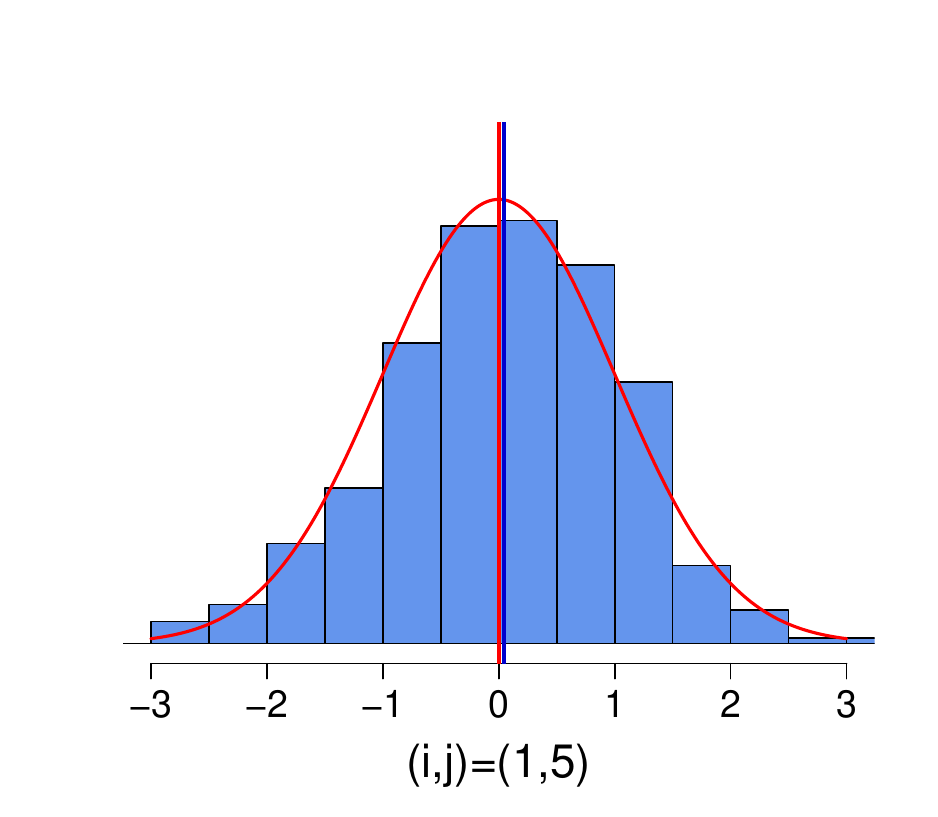}
    \end{minipage}
    \begin{minipage}{0.24\linewidth}
        \centering
        \includegraphics[width=\textwidth]{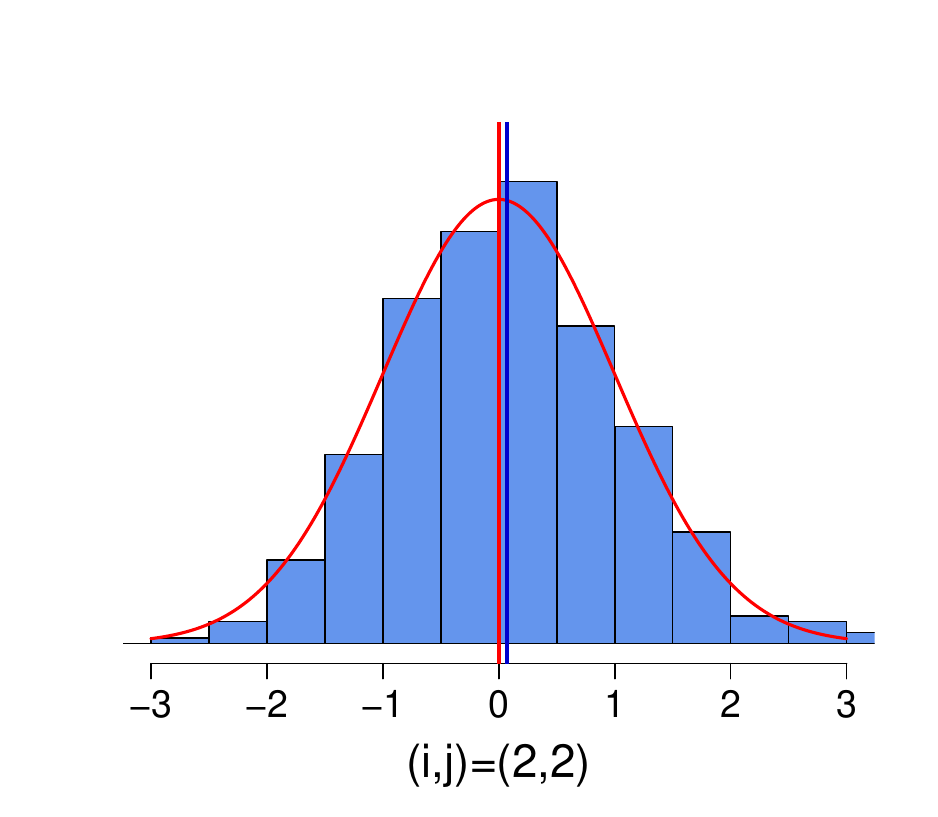}
    \end{minipage}
 \end{minipage} 
     \hspace{1cm}
 \begin{minipage}{0.3\linewidth}
    \begin{minipage}{0.24\linewidth}
        \centering
        \includegraphics[width=\textwidth]{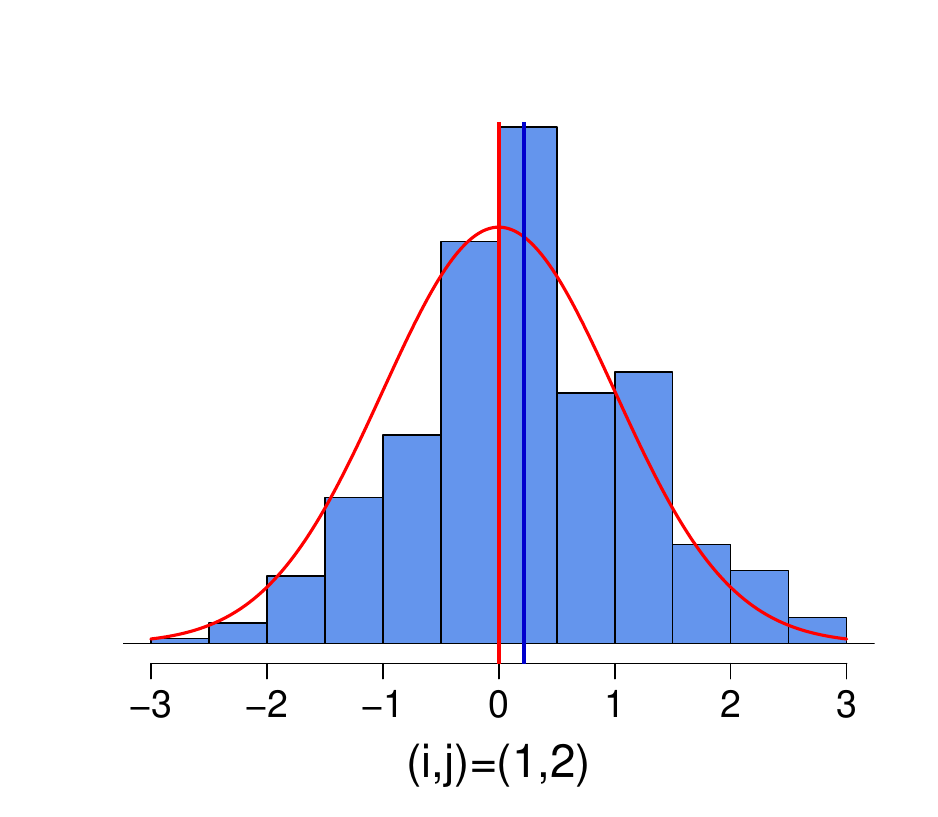}
    \end{minipage}
    \begin{minipage}{0.24\linewidth}
        \centering
        \includegraphics[width=\textwidth]{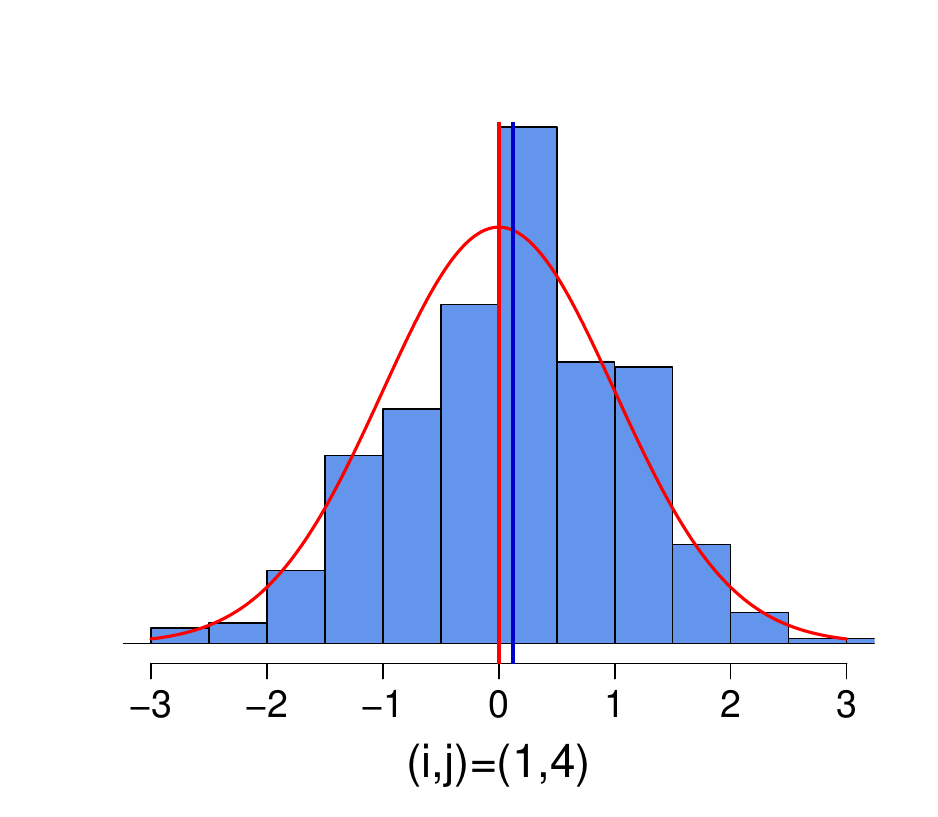}
    \end{minipage}
    \begin{minipage}{0.24\linewidth}
        \centering
        \includegraphics[width=\textwidth]{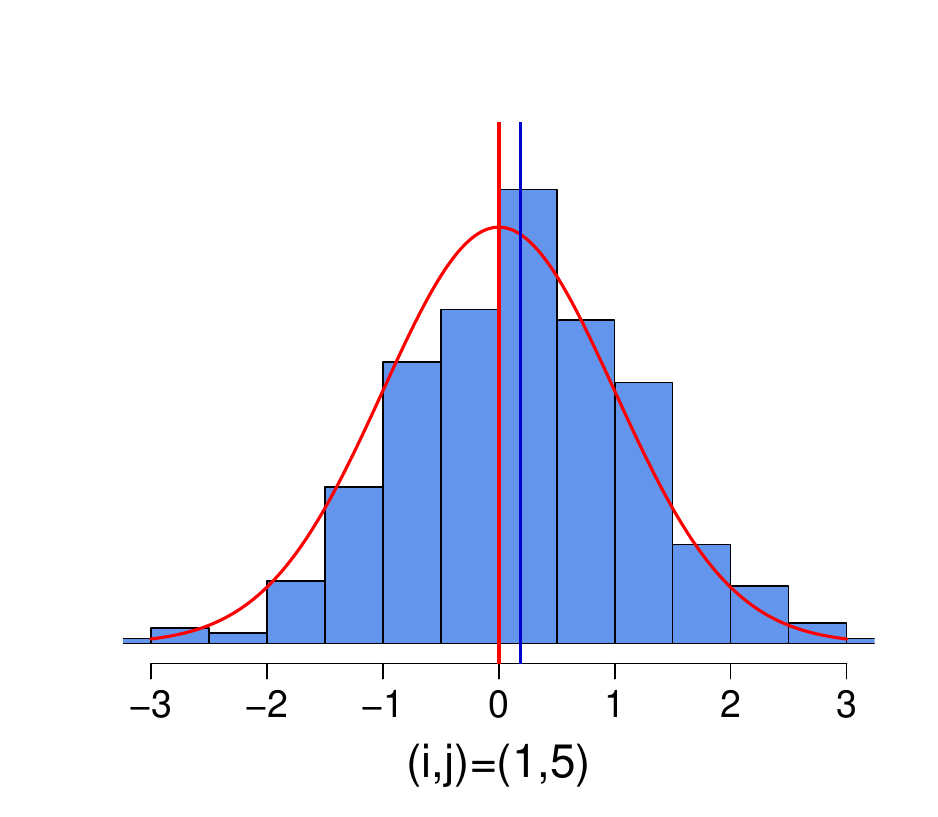}
    \end{minipage}
    \begin{minipage}{0.24\linewidth}
        \centering
        \includegraphics[width=\textwidth]{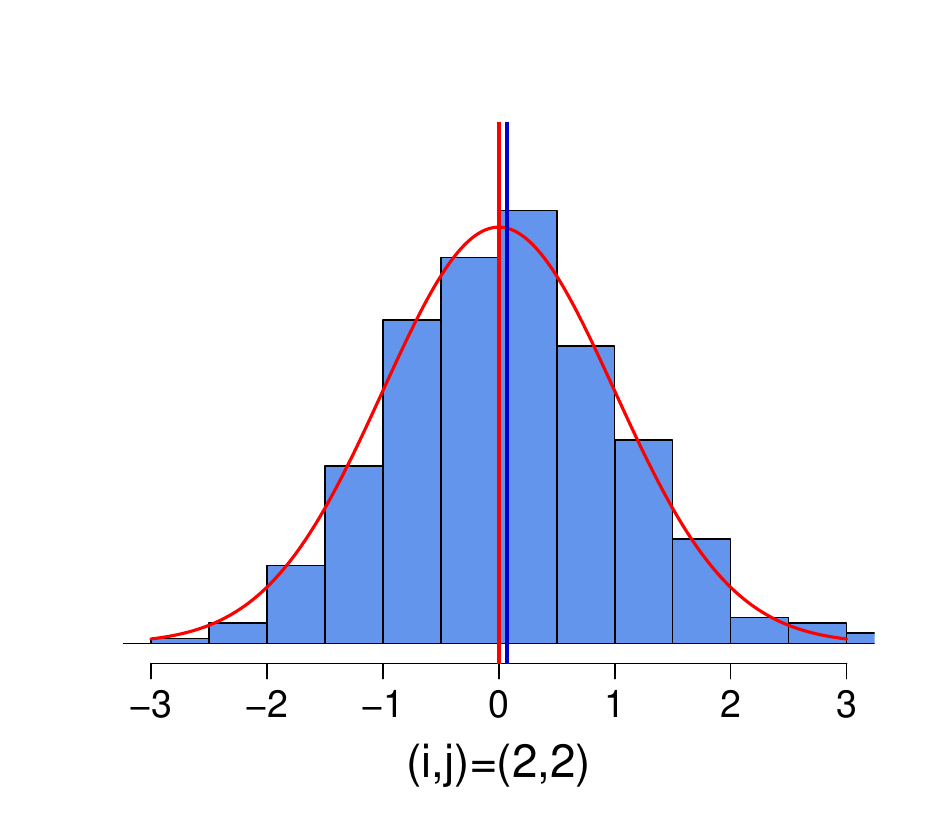}
    \end{minipage}
 \end{minipage}   
      \hspace{1cm}
 \begin{minipage}{0.3\linewidth}
    \begin{minipage}{0.24\linewidth}
        \centering
        \includegraphics[width=\textwidth]{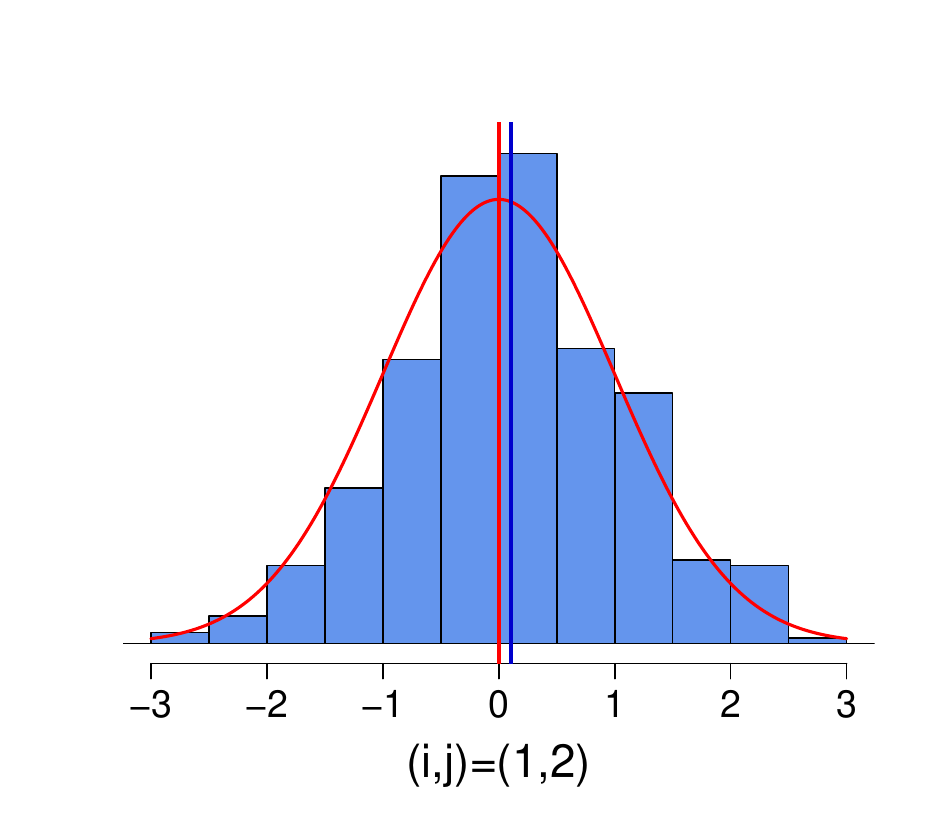}
    \end{minipage}
    \begin{minipage}{0.24\linewidth}
        \centering
        \includegraphics[width=\textwidth]{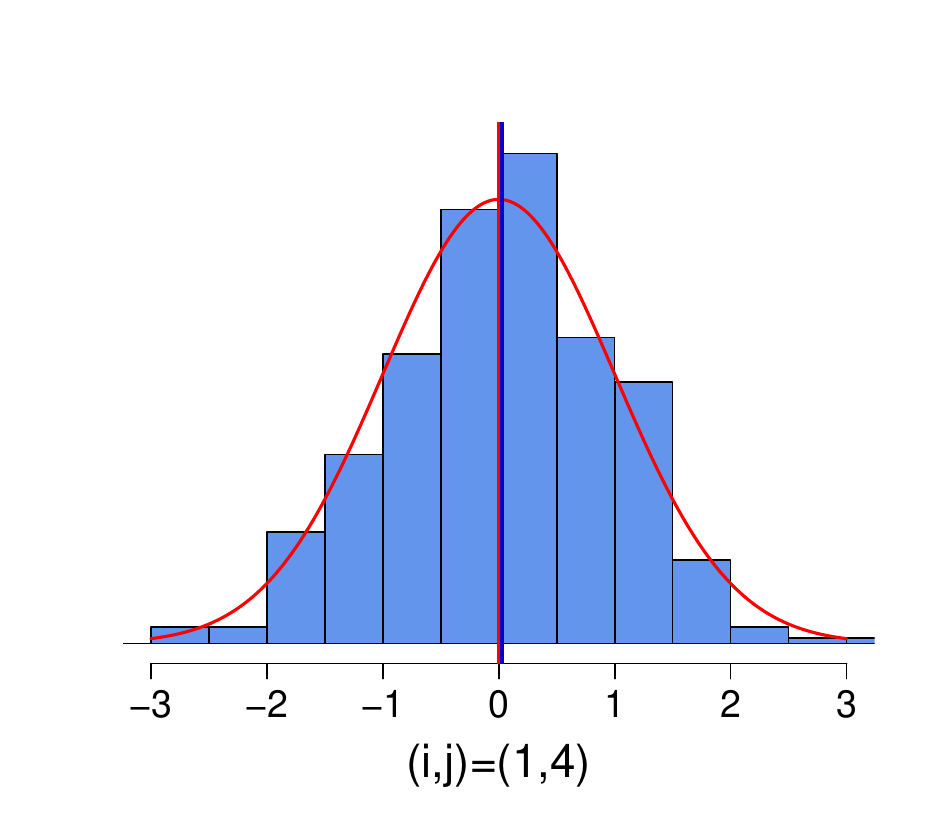}
    \end{minipage}
    \begin{minipage}{0.24\linewidth}
        \centering
        \includegraphics[width=\textwidth]{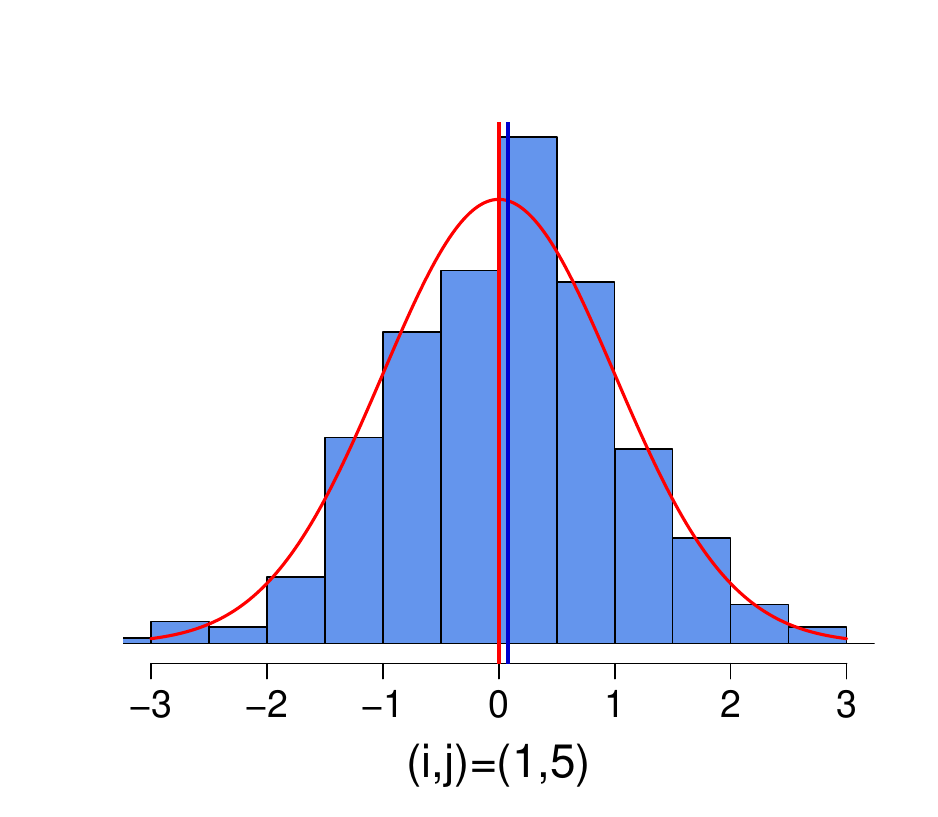}
    \end{minipage}
    \begin{minipage}{0.24\linewidth}
        \centering
        \includegraphics[width=\textwidth]{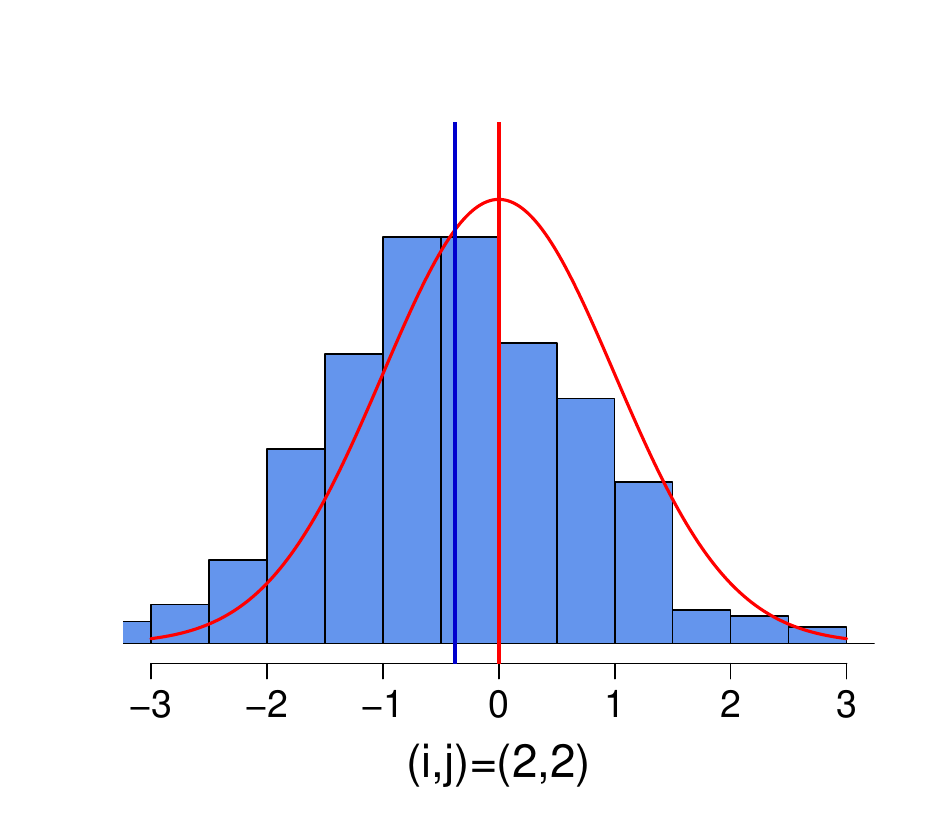}
    \end{minipage}
     \end{minipage}   
     
 \caption*{$n=200, p=400$}
     \vspace{-0.43cm}
 \begin{minipage}{0.3\linewidth}
    \begin{minipage}{0.24\linewidth}
        \centering
        \includegraphics[width=\textwidth]{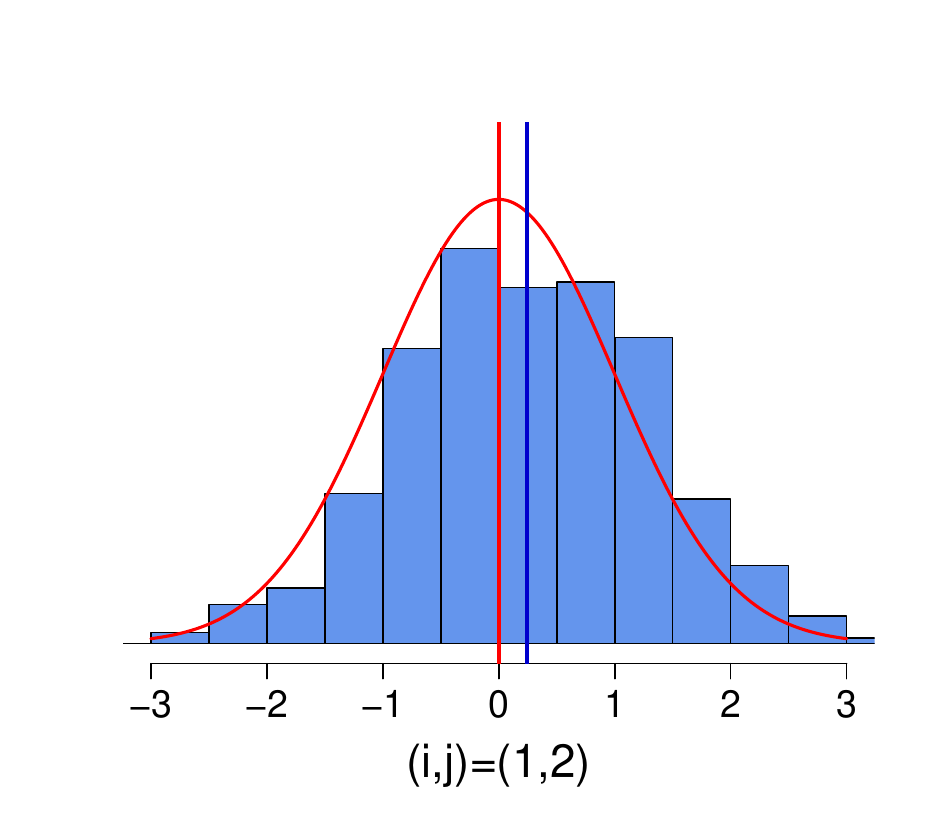}
    \end{minipage}
    \begin{minipage}{0.24\linewidth}
        \centering
        \includegraphics[width=\textwidth]{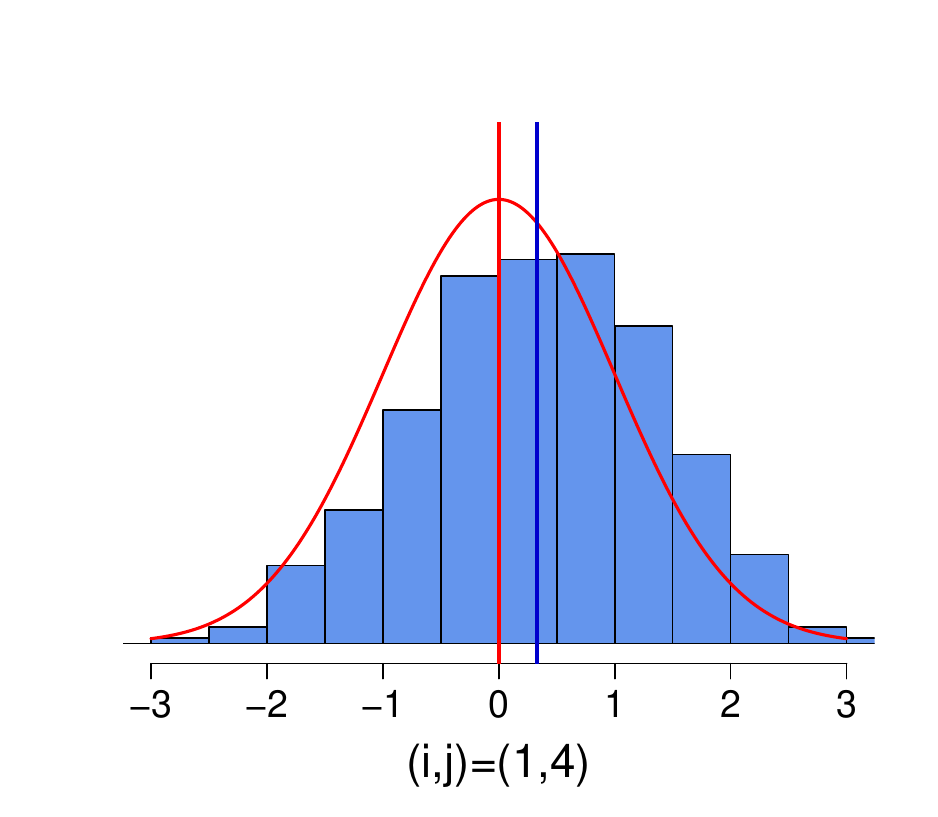}
    \end{minipage}
    \begin{minipage}{0.24\linewidth}
        \centering
        \includegraphics[width=\textwidth]{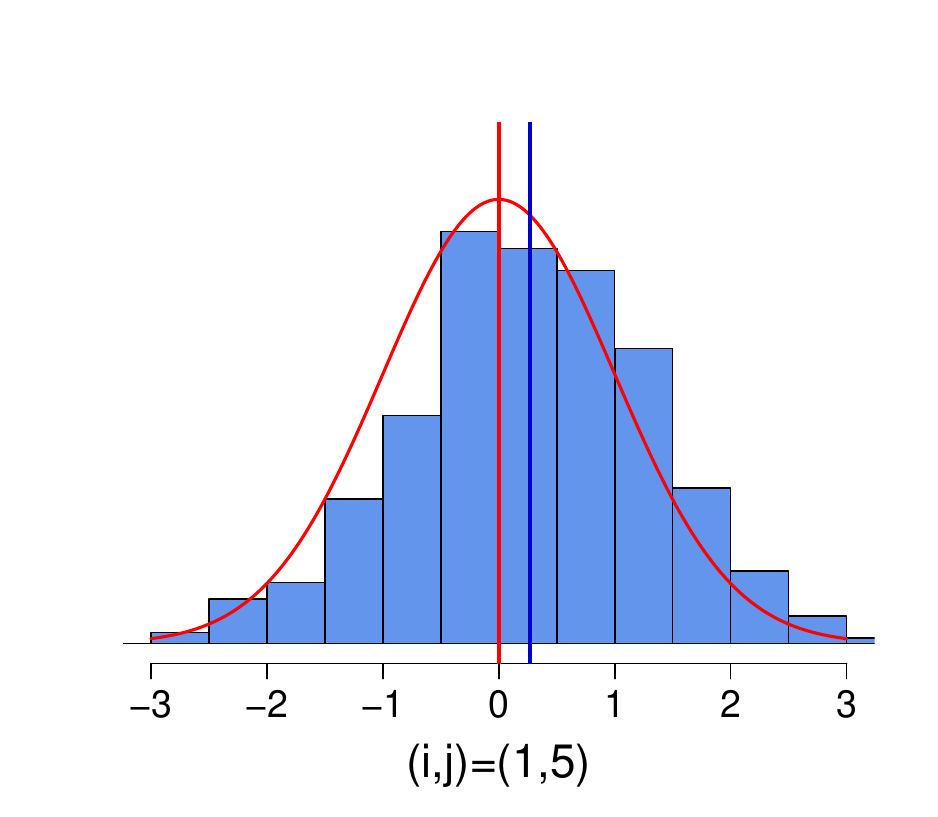}
    \end{minipage}
    \begin{minipage}{0.24\linewidth}
        \centering
        \includegraphics[width=\textwidth]{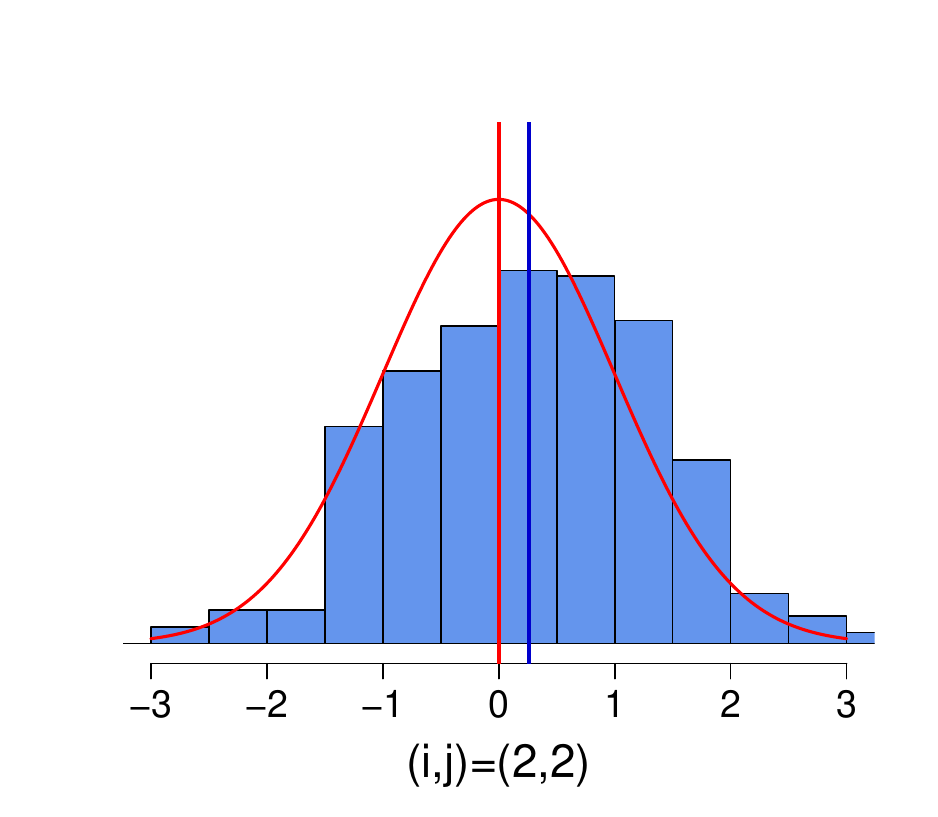}
    \end{minipage}
 \end{minipage}
 \hspace{1cm}
 \begin{minipage}{0.3\linewidth}
    \begin{minipage}{0.24\linewidth}
        \centering
        \includegraphics[width=\textwidth]{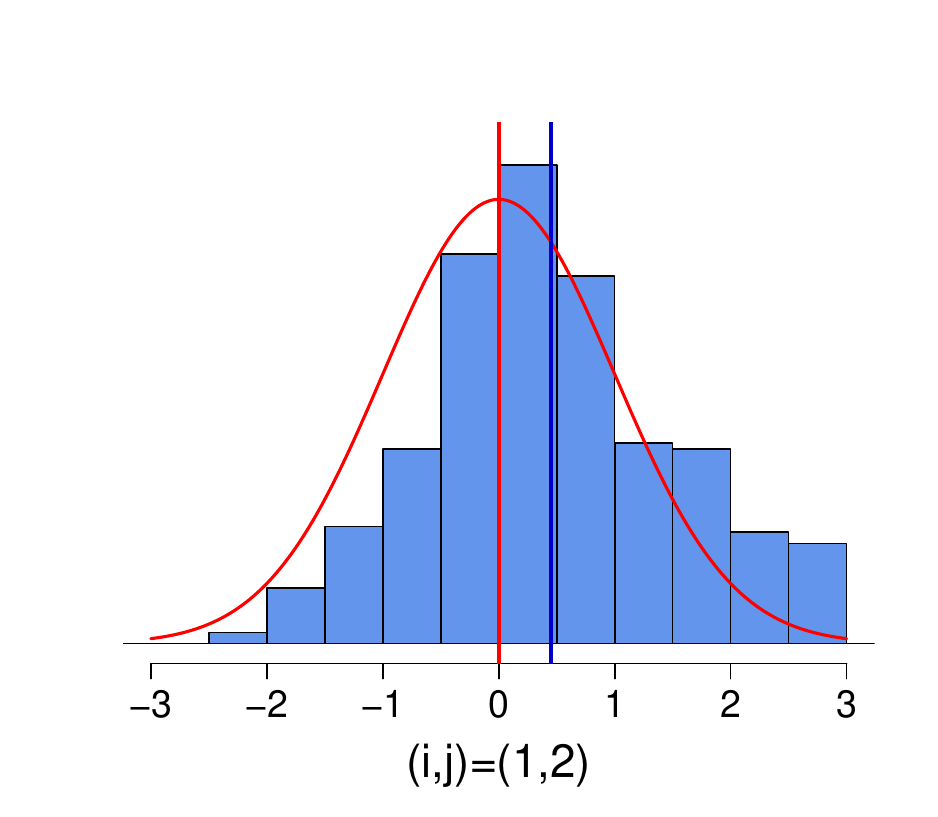}
    \end{minipage}
    \begin{minipage}{0.24\linewidth}
        \centering
        \includegraphics[width=\textwidth]{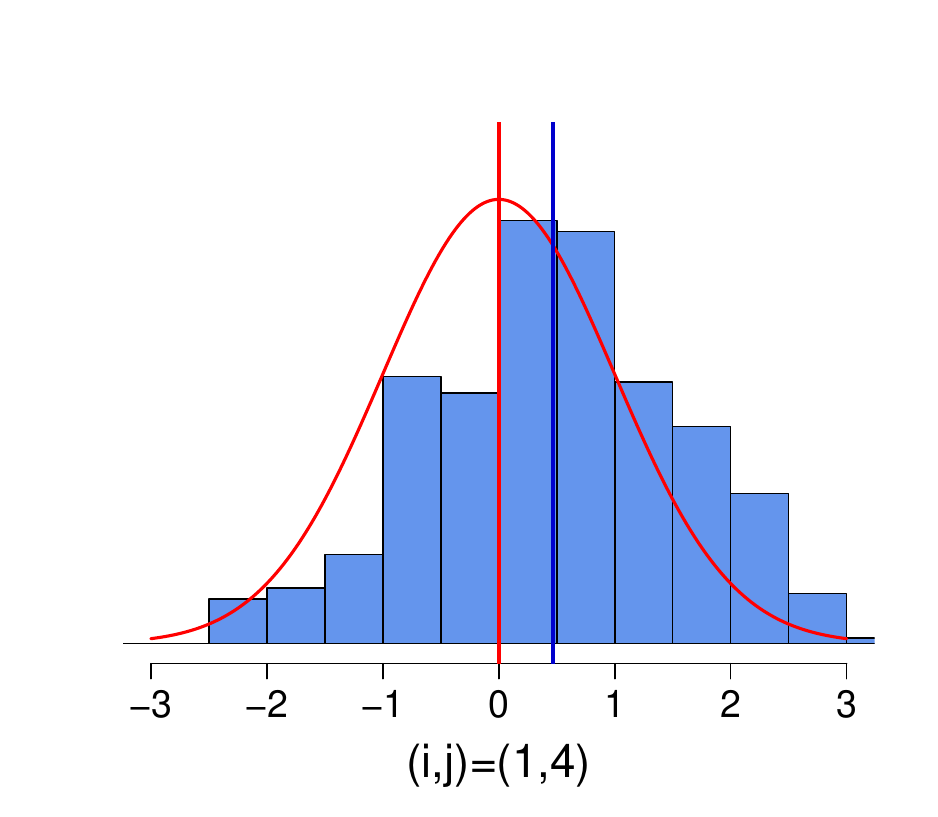}
    \end{minipage}
    \begin{minipage}{0.24\linewidth}
        \centering
        \includegraphics[width=\textwidth]{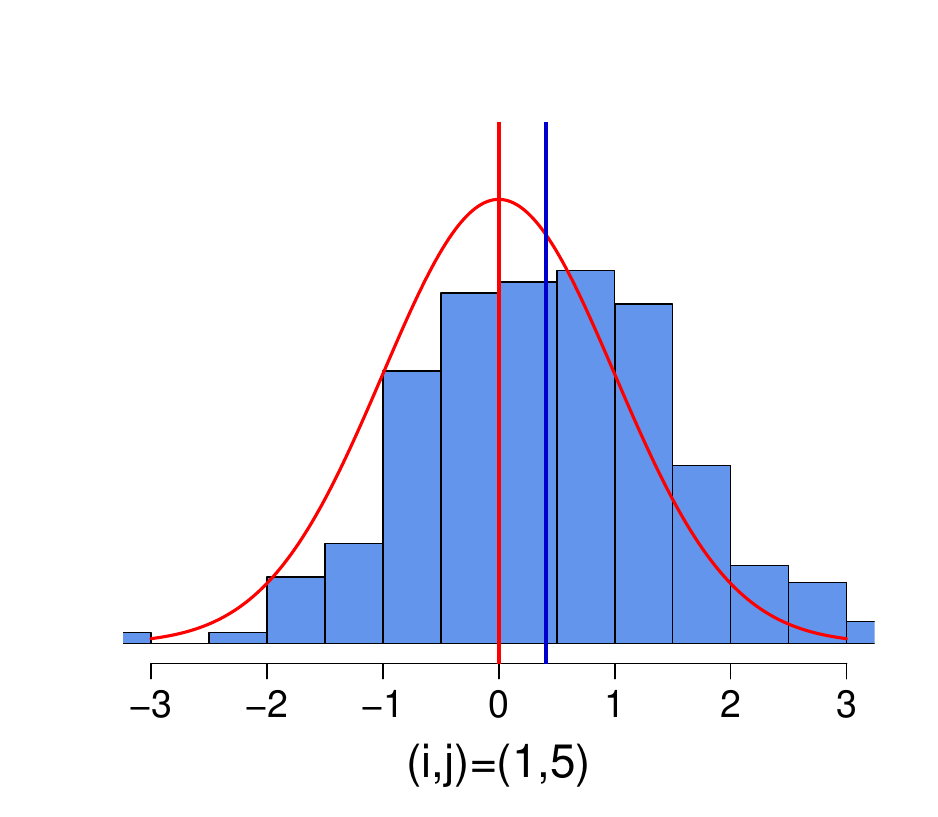}
    \end{minipage}
    \begin{minipage}{0.24\linewidth}
        \centering
        \includegraphics[width=\textwidth]{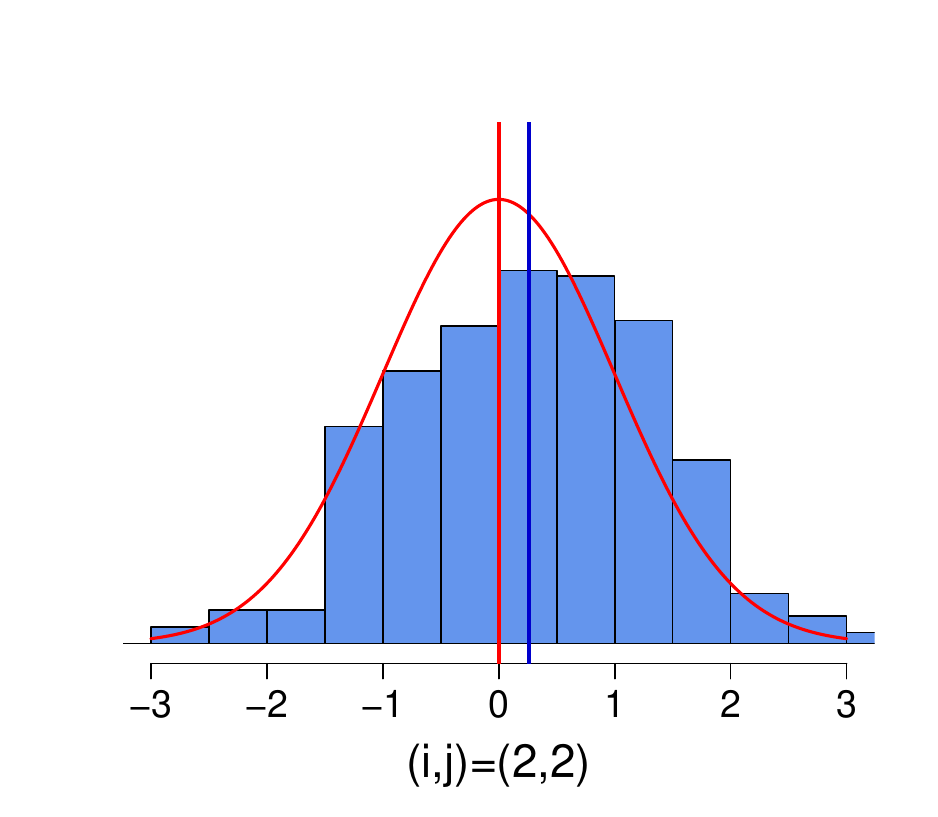}
    \end{minipage}    
 \end{minipage}
  \hspace{1cm}
 \begin{minipage}{0.3\linewidth}
     \begin{minipage}{0.24\linewidth}
        \centering
        \includegraphics[width=\textwidth]{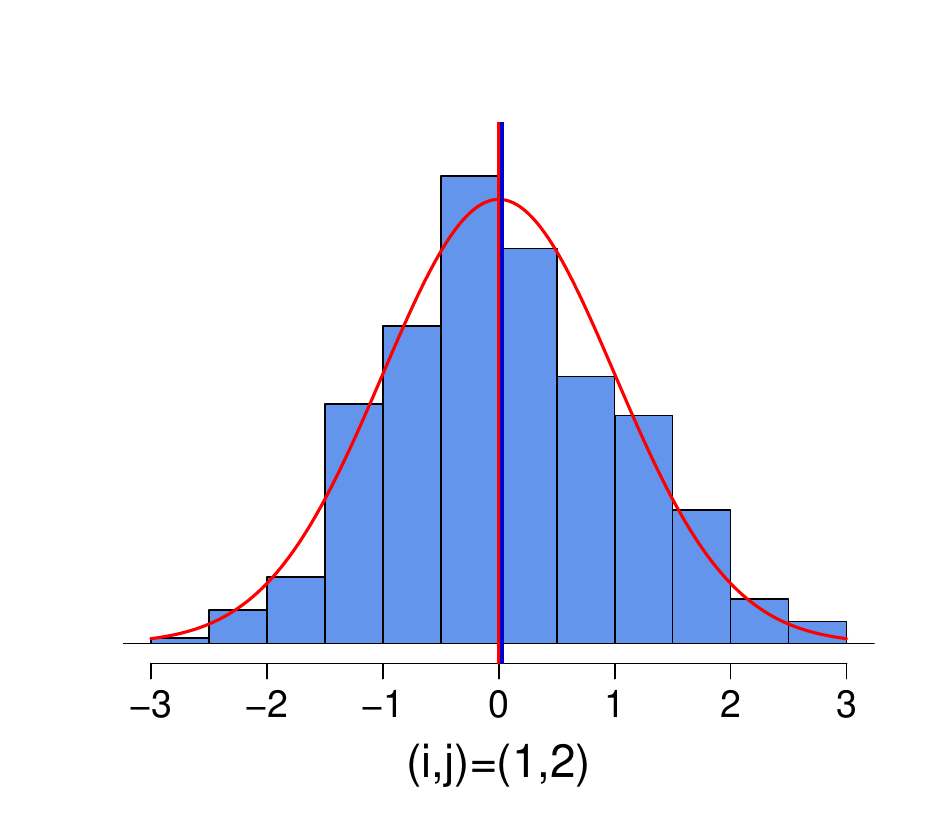}
    \end{minipage}
    \begin{minipage}{0.24\linewidth}
        \centering
        \includegraphics[width=\textwidth]{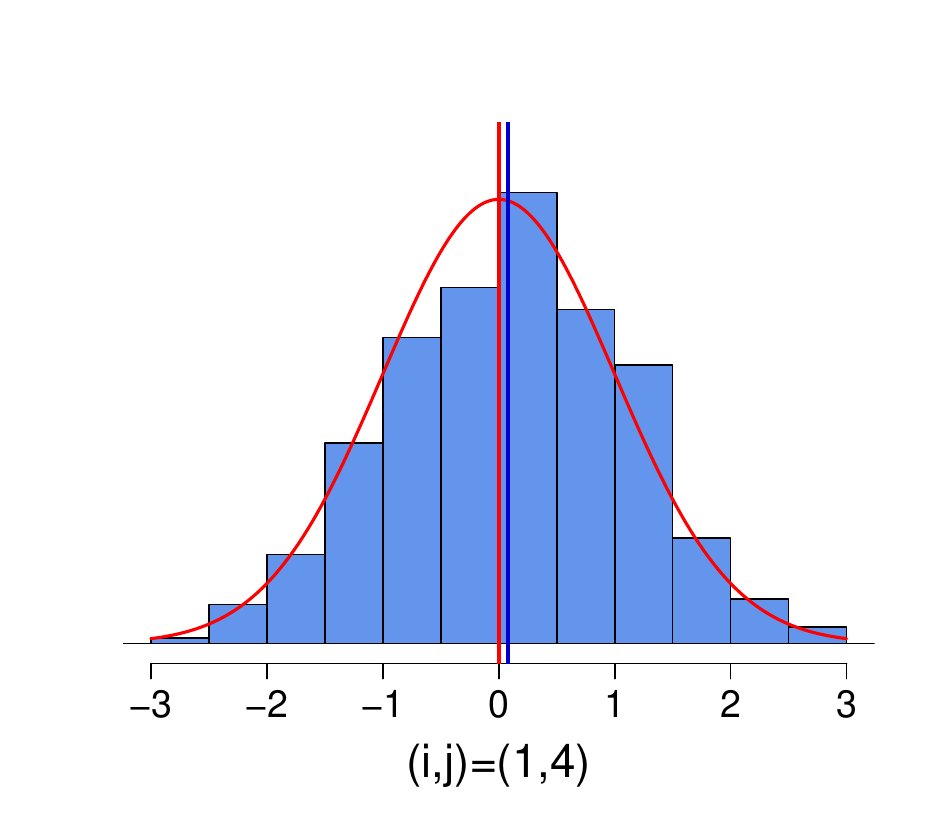}
    \end{minipage}
    \begin{minipage}{0.24\linewidth}
        \centering
        \includegraphics[width=\textwidth]{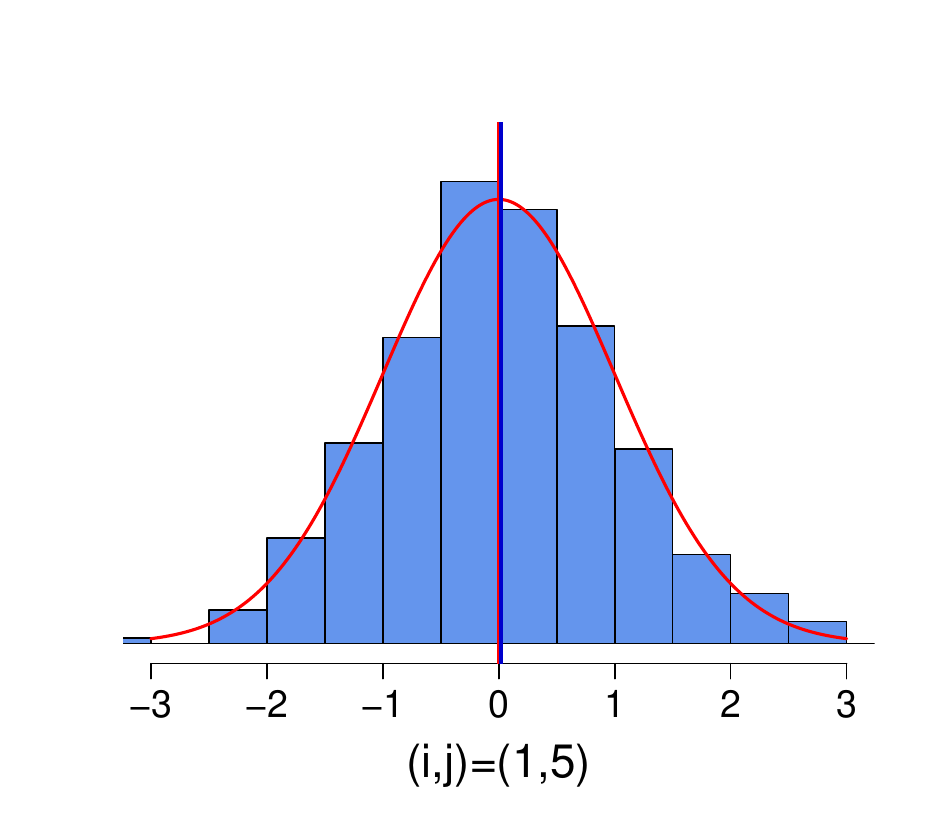}
    \end{minipage}
    \begin{minipage}{0.24\linewidth}
        \centering
        \includegraphics[width=\textwidth]{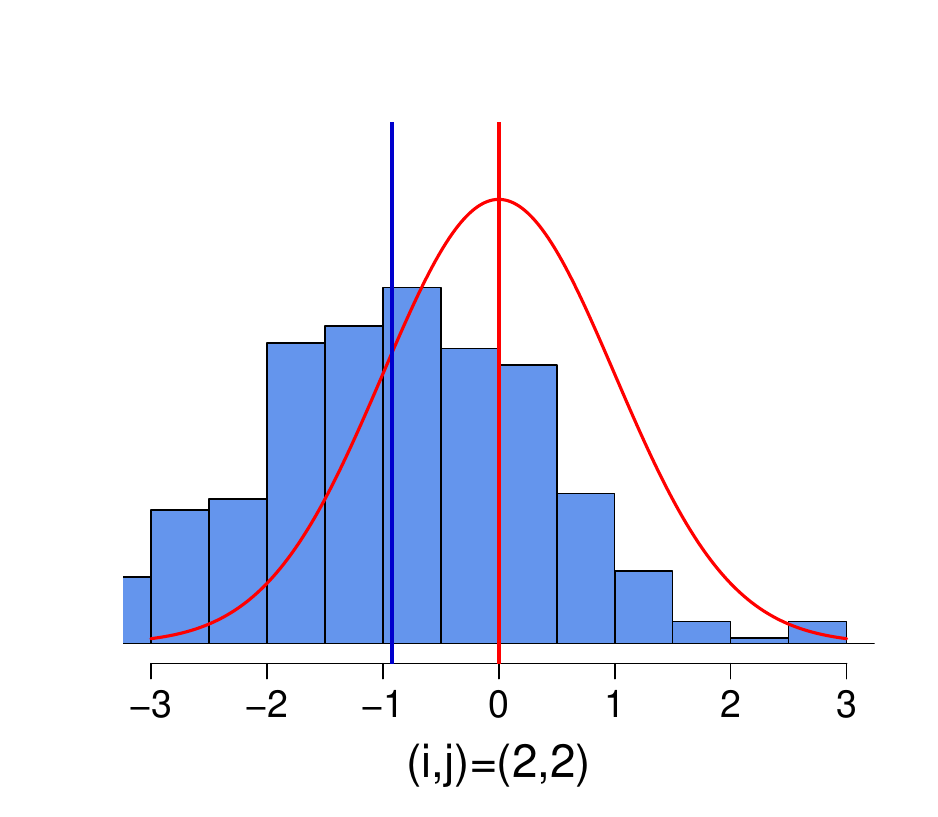}
    \end{minipage}
 \end{minipage}

  \caption*{$n=400, p=400$}
      \vspace{-0.43cm}
 \begin{minipage}{0.3\linewidth}
    \begin{minipage}{0.24\linewidth}
        \centering
        \includegraphics[width=\textwidth]{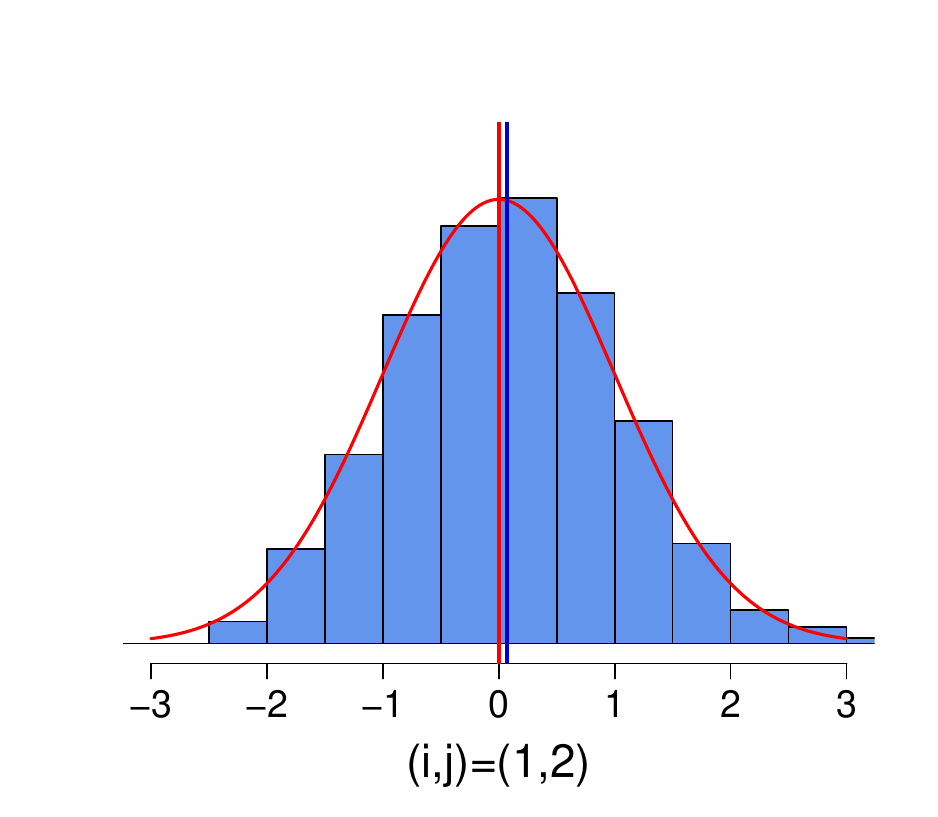}
    \end{minipage}
    \begin{minipage}{0.24\linewidth}
        \centering
        \includegraphics[width=\textwidth]{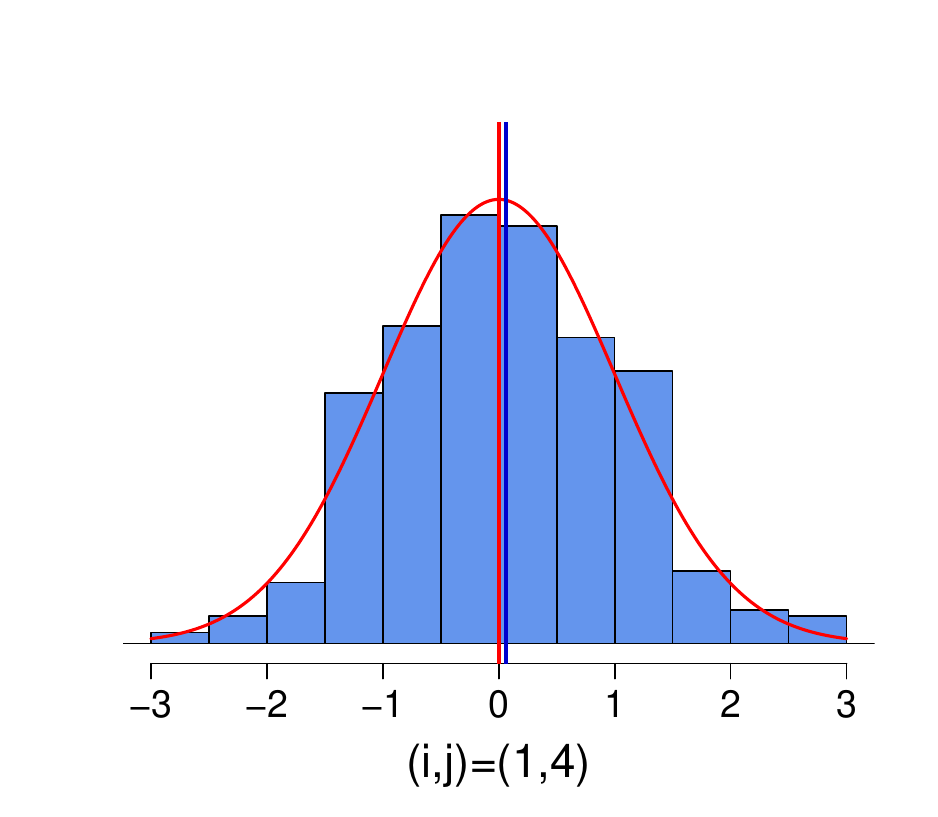}
    \end{minipage}
    \begin{minipage}{0.24\linewidth}
        \centering
        \includegraphics[width=\textwidth]{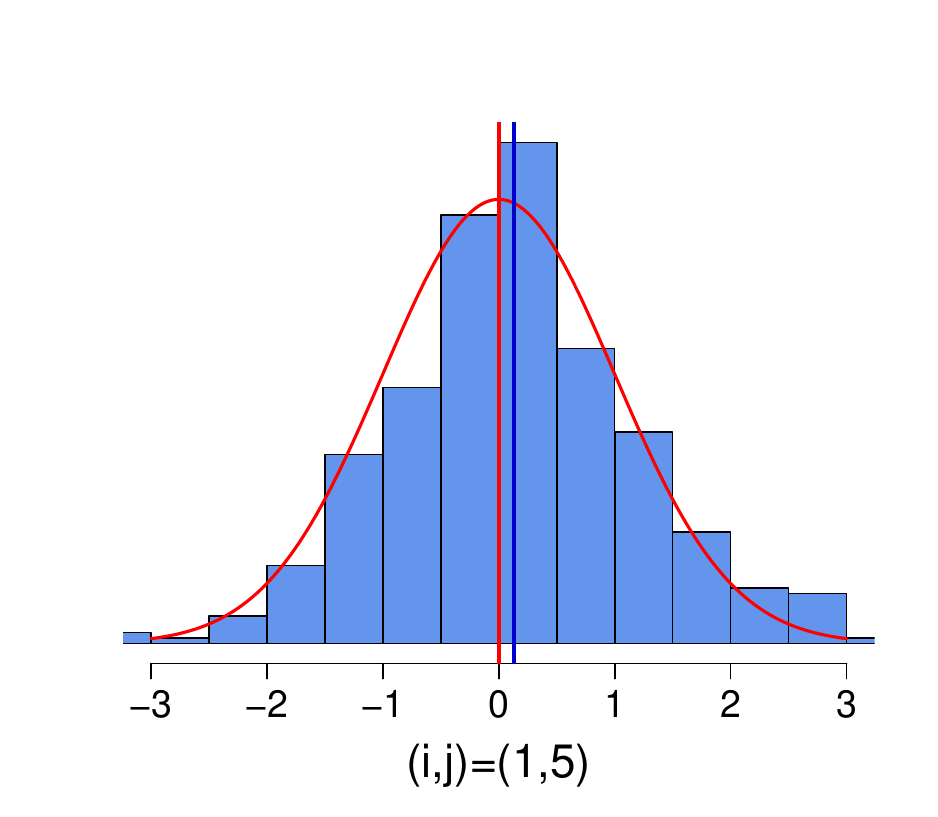}
    \end{minipage}
    \begin{minipage}{0.24\linewidth}
        \centering
        \includegraphics[width=\textwidth]{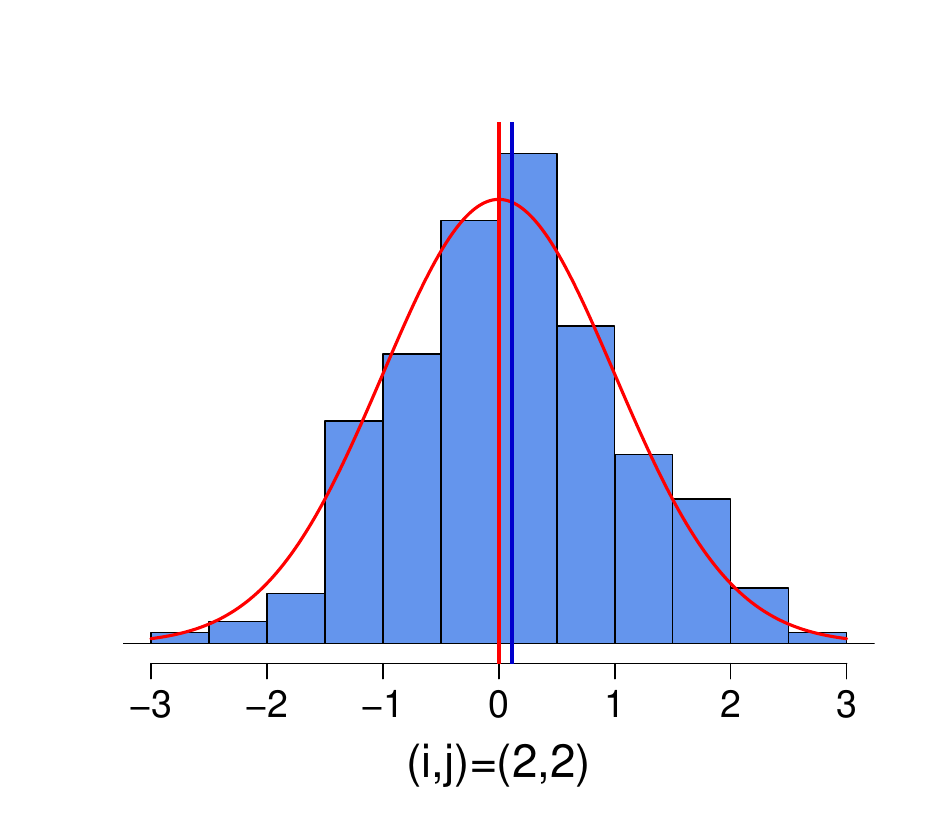}
    \end{minipage}
 \end{minipage}  
     \hspace{1cm}
 \begin{minipage}{0.3\linewidth}
    \begin{minipage}{0.24\linewidth}
        \centering
        \includegraphics[width=\textwidth]{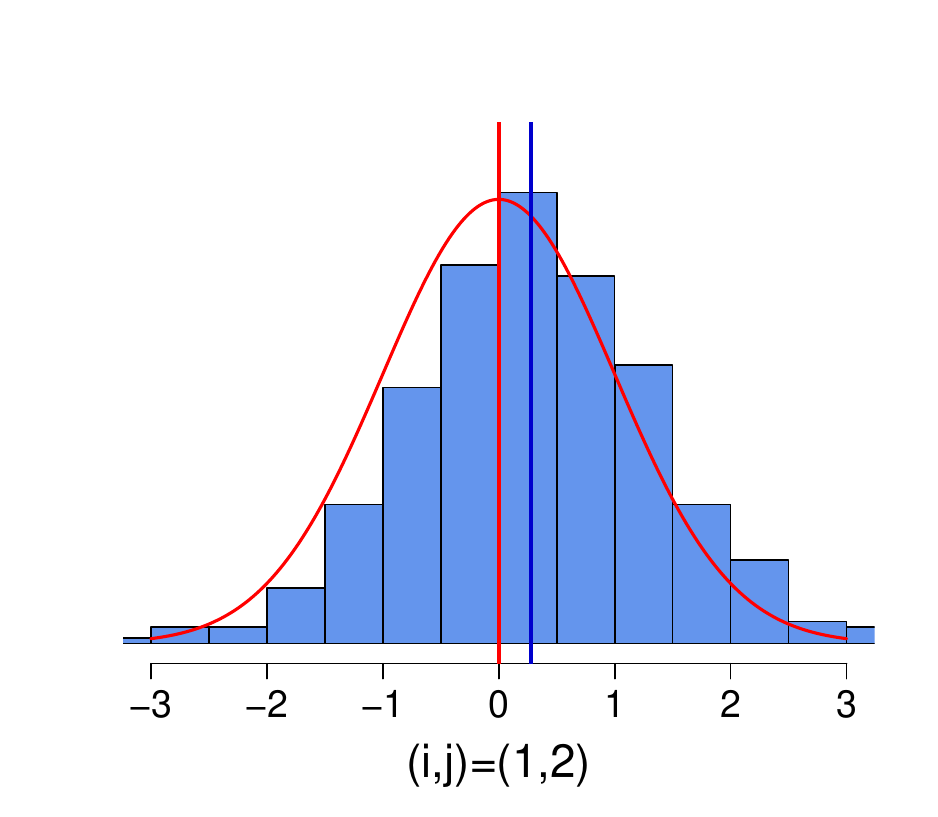}
    \end{minipage}
    \begin{minipage}{0.24\linewidth}
        \centering
        \includegraphics[width=\textwidth]{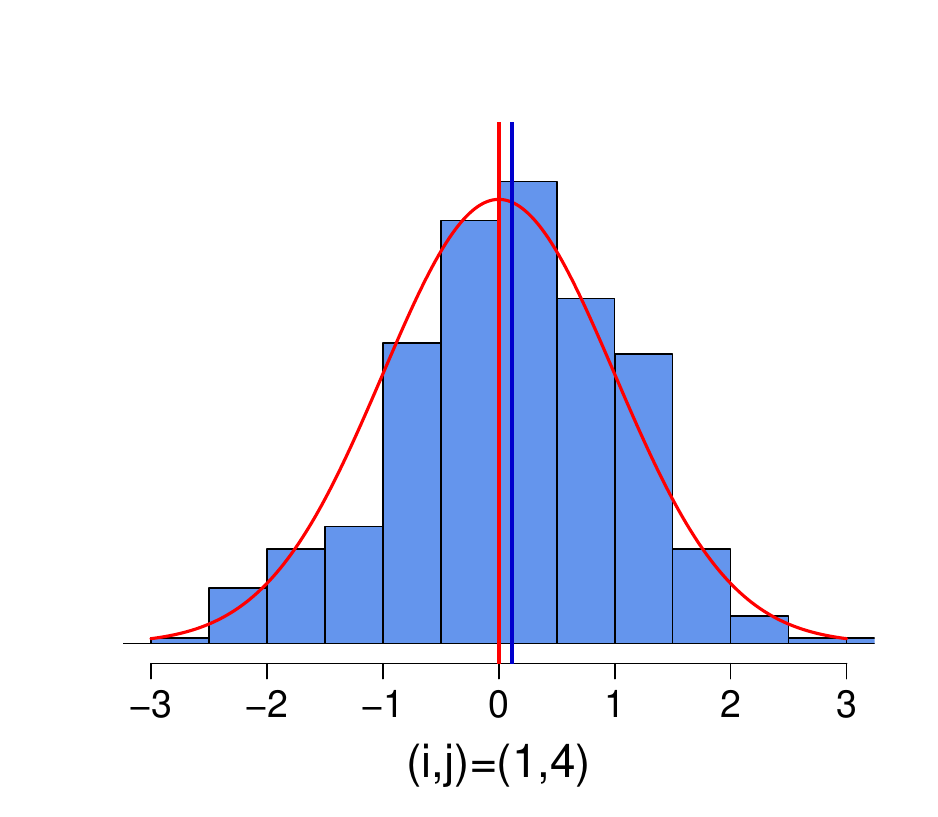}
    \end{minipage}
    \begin{minipage}{0.24\linewidth}
        \centering
        \includegraphics[width=\textwidth]{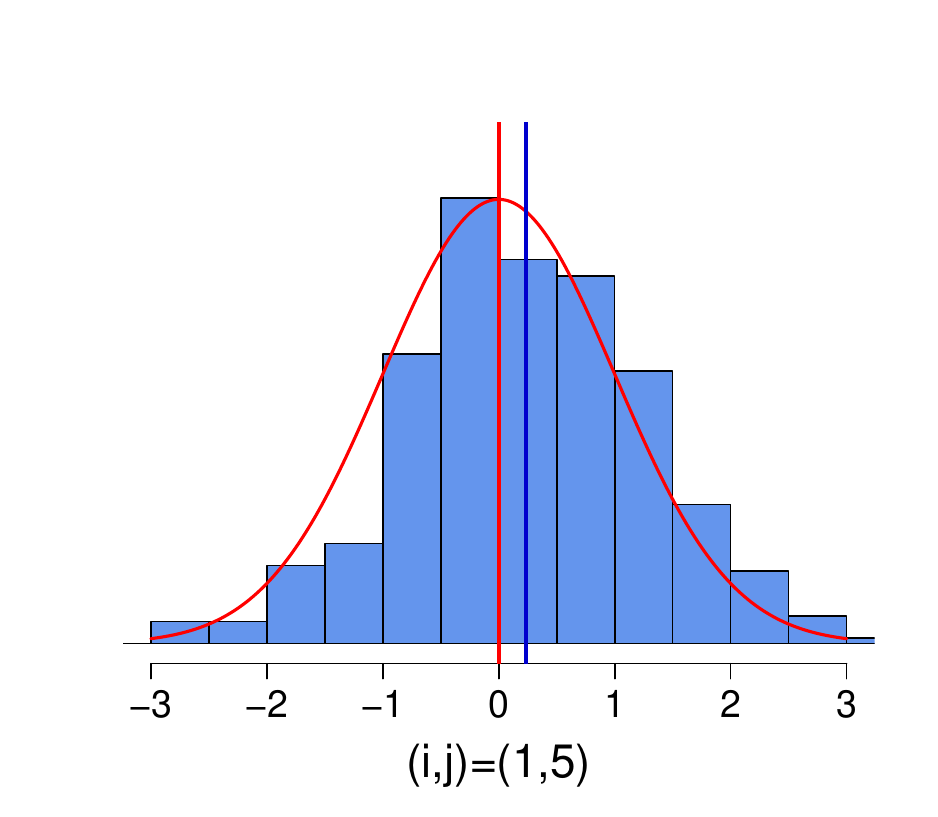}
    \end{minipage}
    \begin{minipage}{0.24\linewidth}
        \centering
        \includegraphics[width=\textwidth]{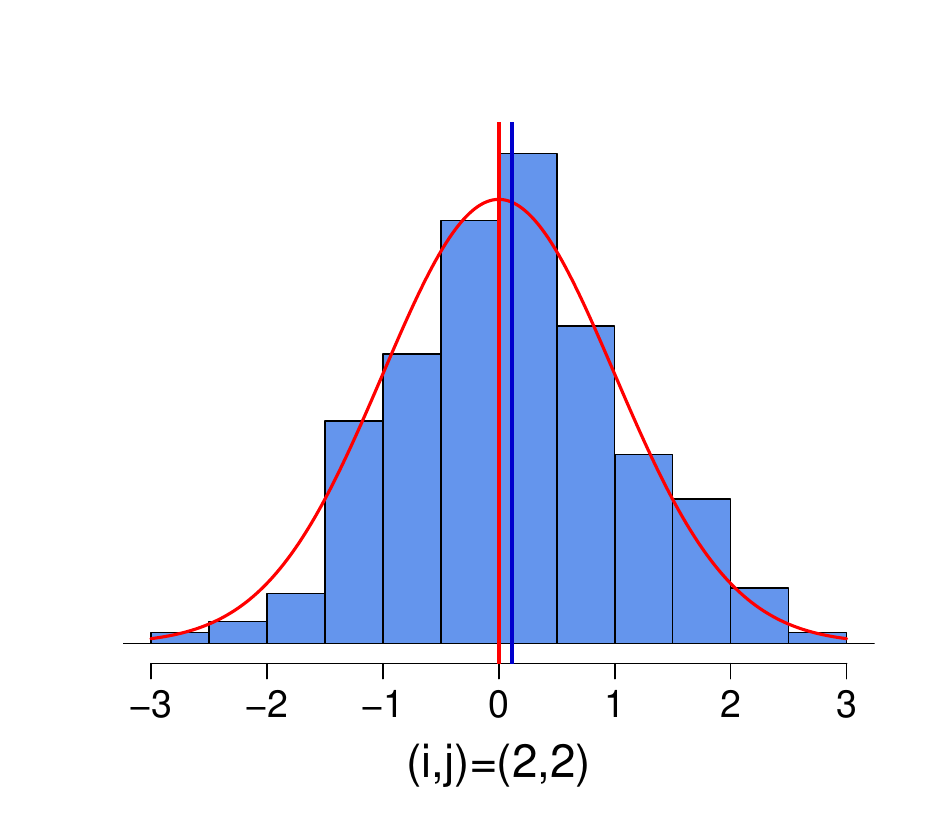}
    \end{minipage}
  \end{minipage}  
    \hspace{1cm}
 \begin{minipage}{0.3\linewidth}
    \begin{minipage}{0.24\linewidth}
        \centering
        \includegraphics[width=\textwidth]{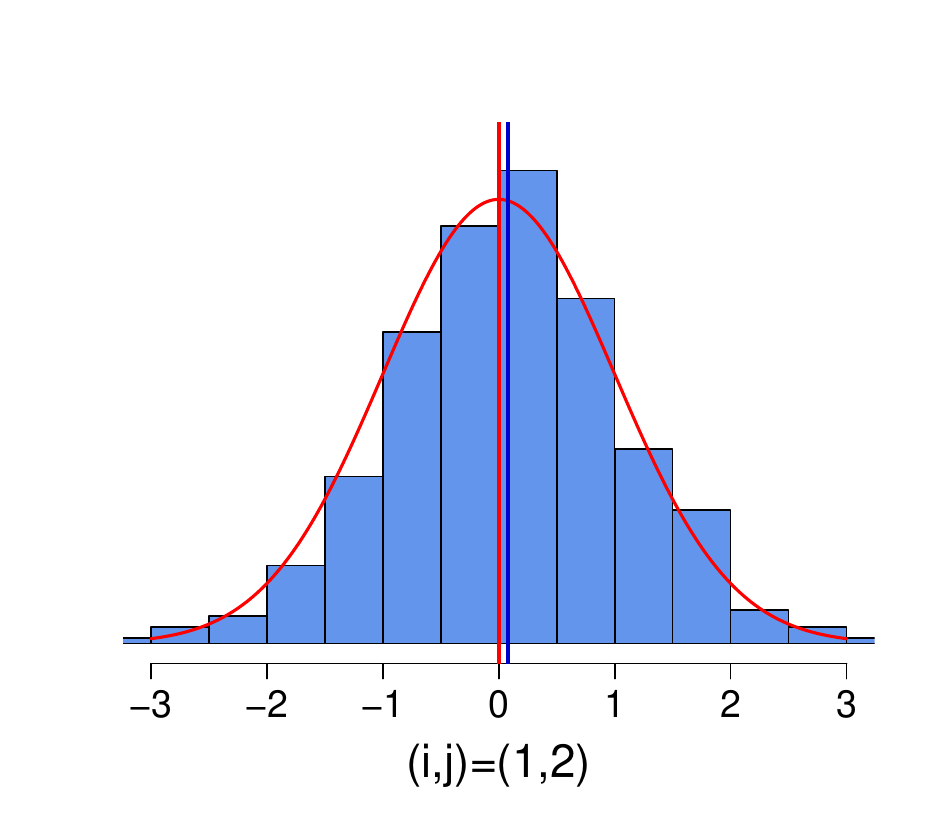}
    \end{minipage}
    \begin{minipage}{0.24\linewidth}
        \centering
        \includegraphics[width=\textwidth]{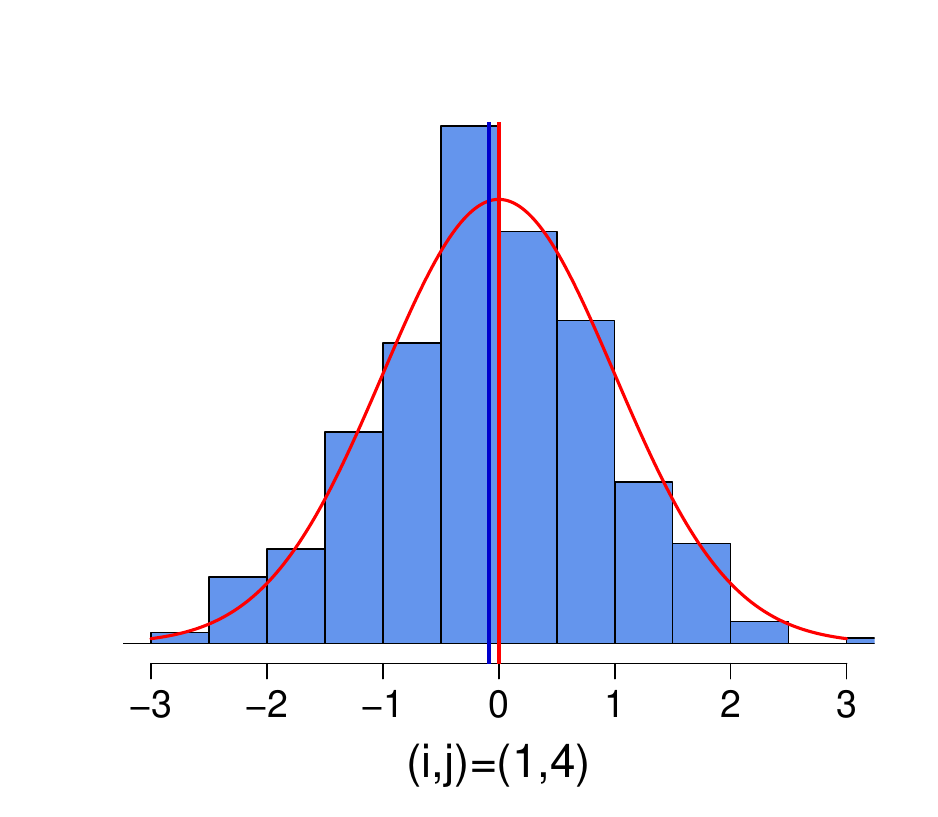}
    \end{minipage}
    \begin{minipage}{0.24\linewidth}
        \centering
        \includegraphics[width=\textwidth]{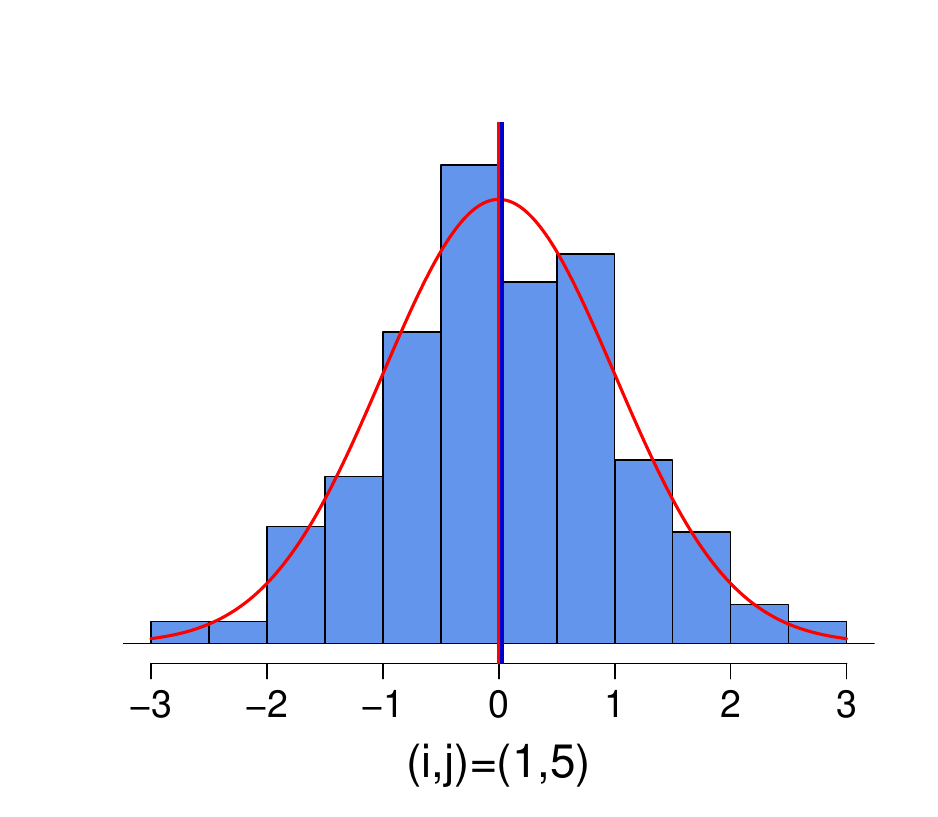}
    \end{minipage}
    \begin{minipage}{0.24\linewidth}
        \centering
        \includegraphics[width=\textwidth]{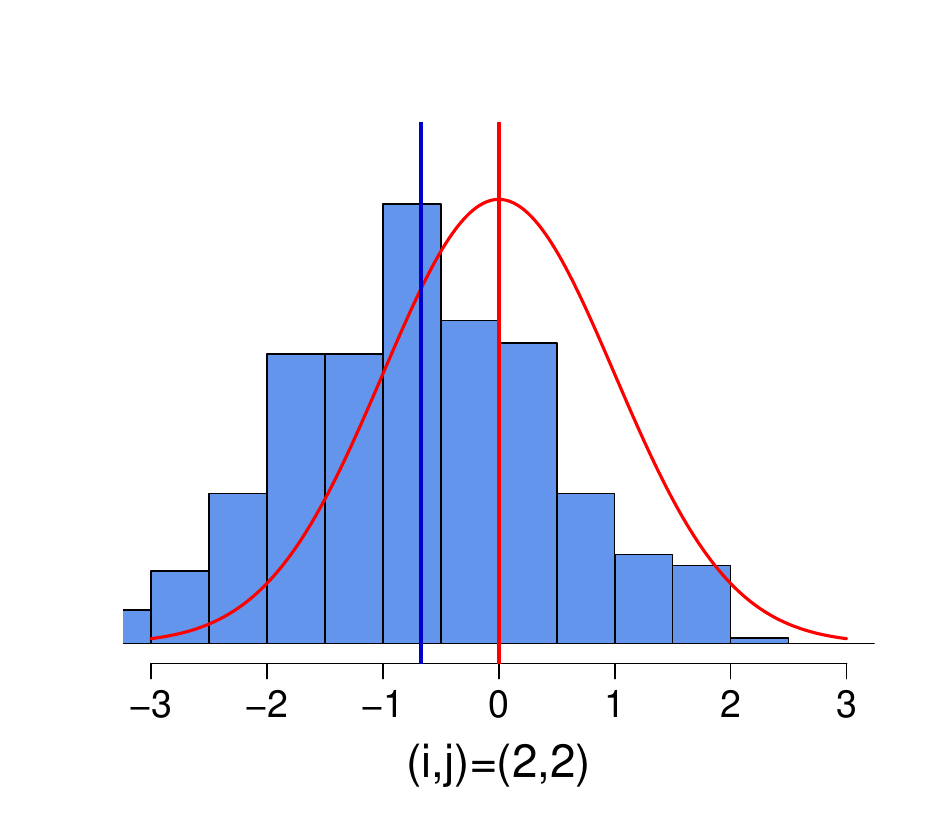}
    \end{minipage}
 \end{minipage}

  \caption*{$n=800, p=400$}
      \vspace{-0.43cm}
 \begin{minipage}{0.3\linewidth}
    \begin{minipage}{0.24\linewidth}
        \centering
        \includegraphics[width=\textwidth]{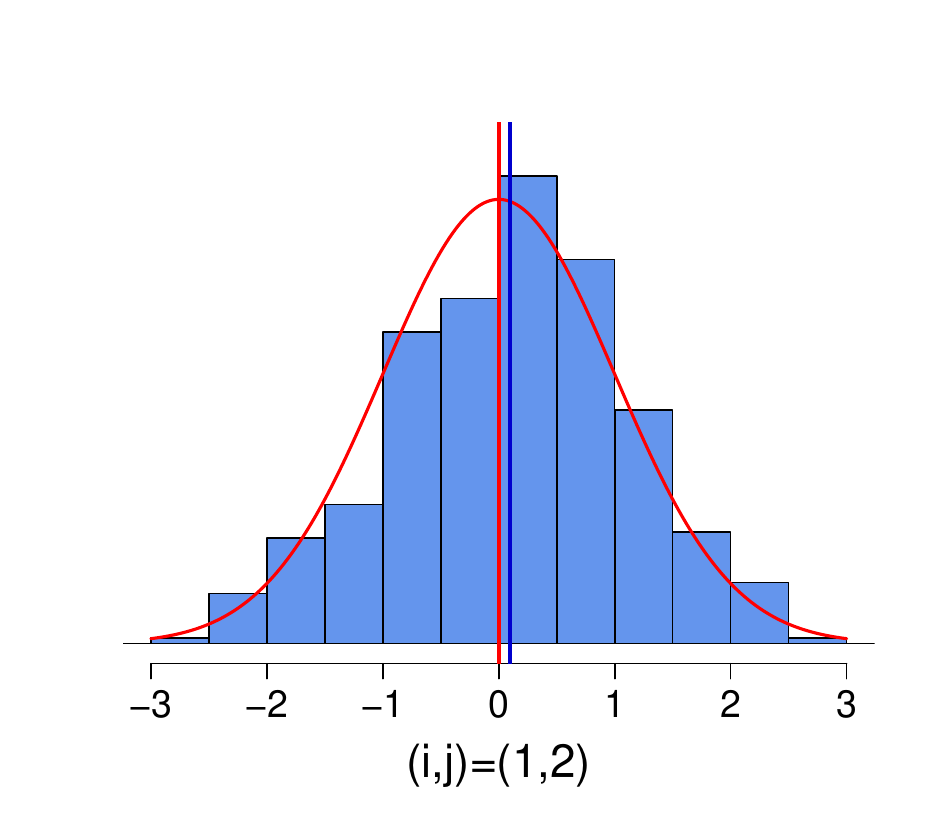}
    \end{minipage}
    \begin{minipage}{0.24\linewidth}
        \centering
        \includegraphics[width=\textwidth]{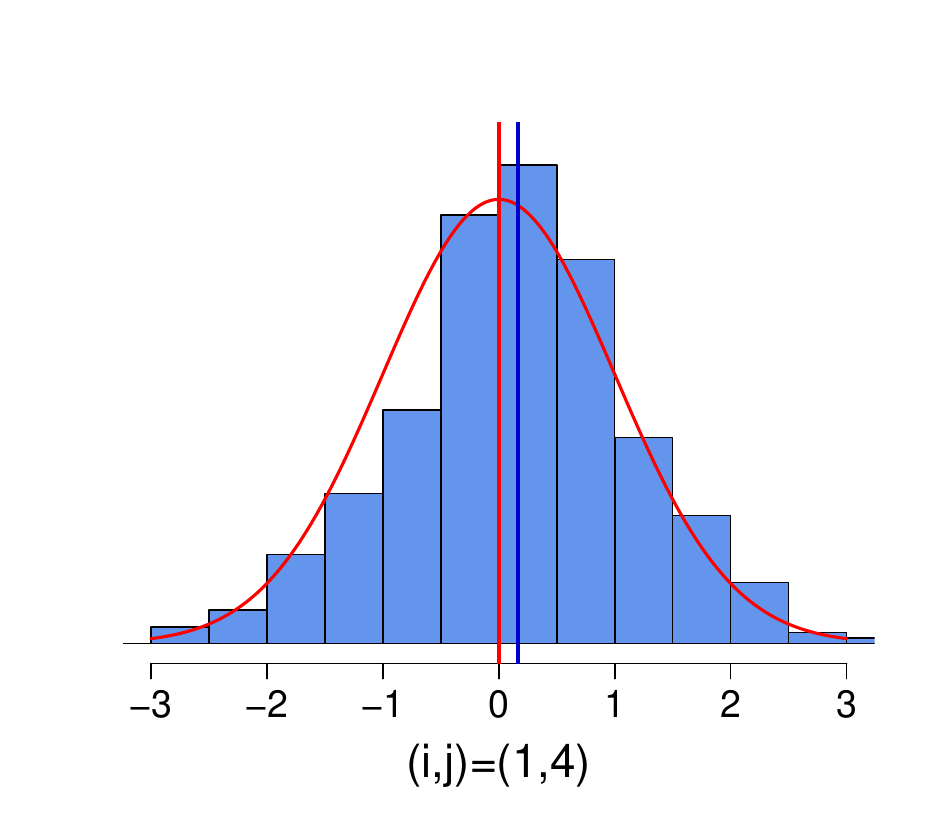}
    \end{minipage}
    \begin{minipage}{0.24\linewidth}
        \centering
        \includegraphics[width=\textwidth]{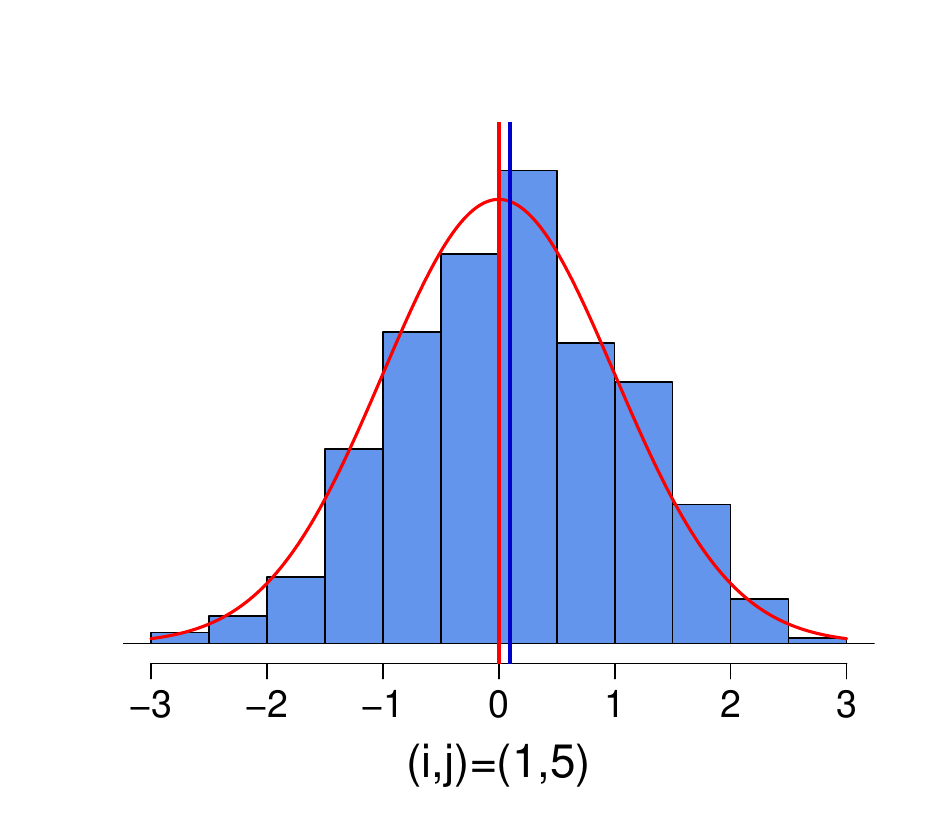}
    \end{minipage}
    \begin{minipage}{0.24\linewidth}
        \centering
        \includegraphics[width=\textwidth]{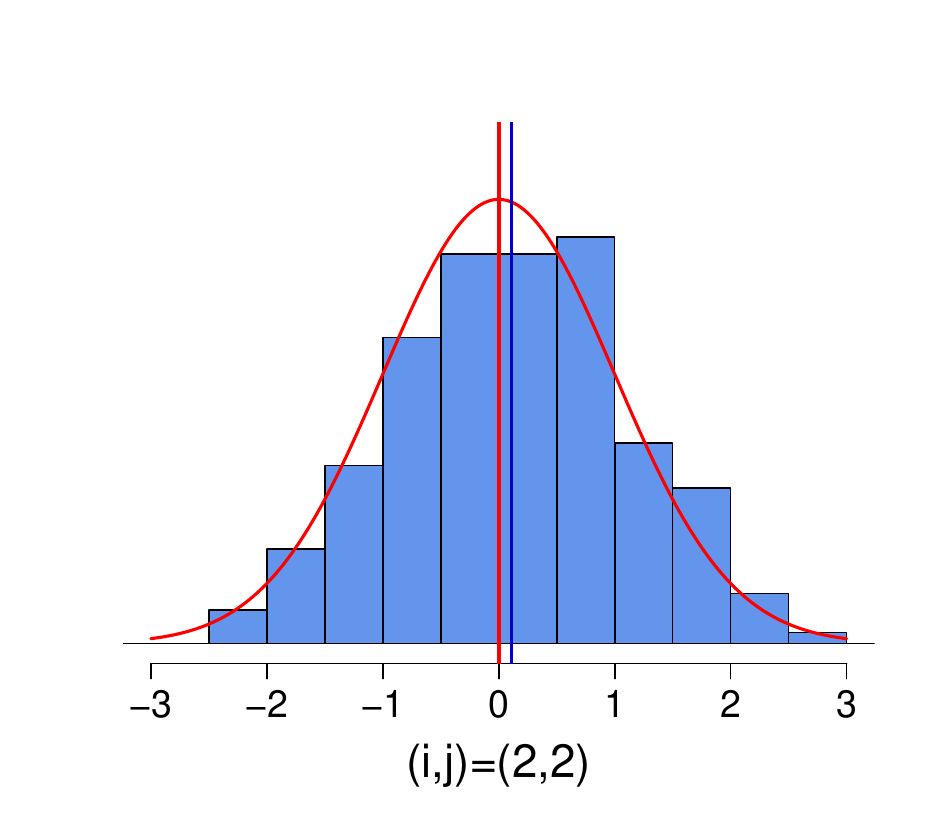}
    \end{minipage}
   \caption*{(a)~~$L_0{:}~ \widehat{\mb{\Omega}}^{\text{US}}$}
 \end{minipage} 
     \hspace{1cm}
 \begin{minipage}{0.3\linewidth}
    \begin{minipage}{0.24\linewidth}
        \centering
        \includegraphics[width=\textwidth]{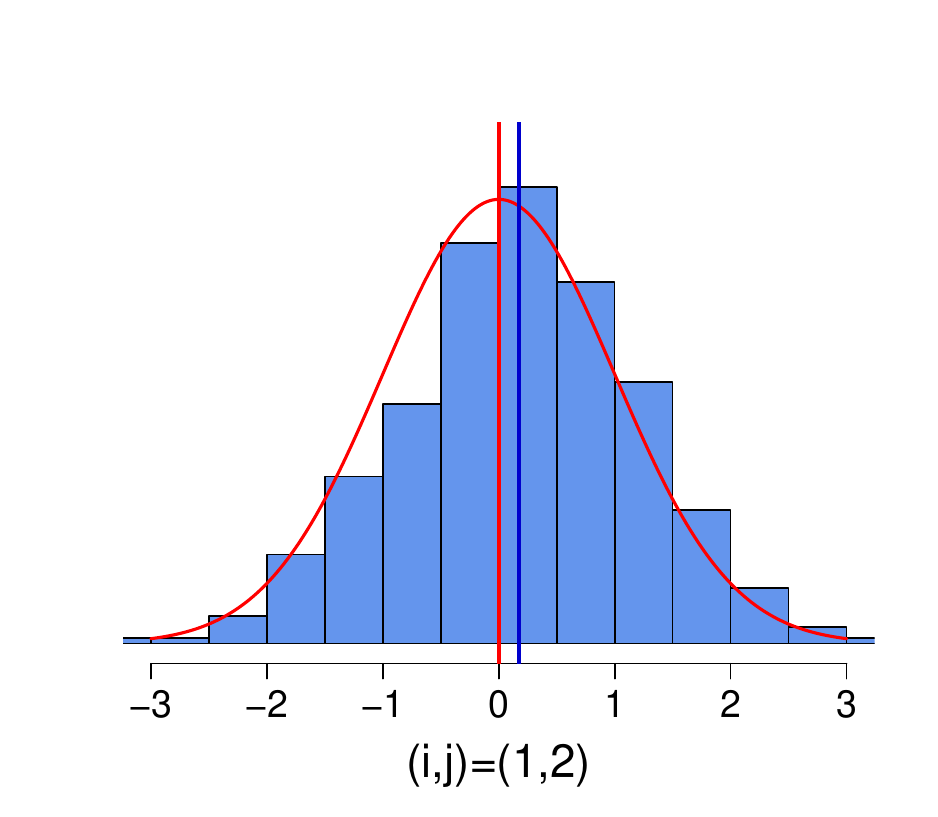}
    \end{minipage}
    \begin{minipage}{0.24\linewidth}
        \centering
        \includegraphics[width=\textwidth]{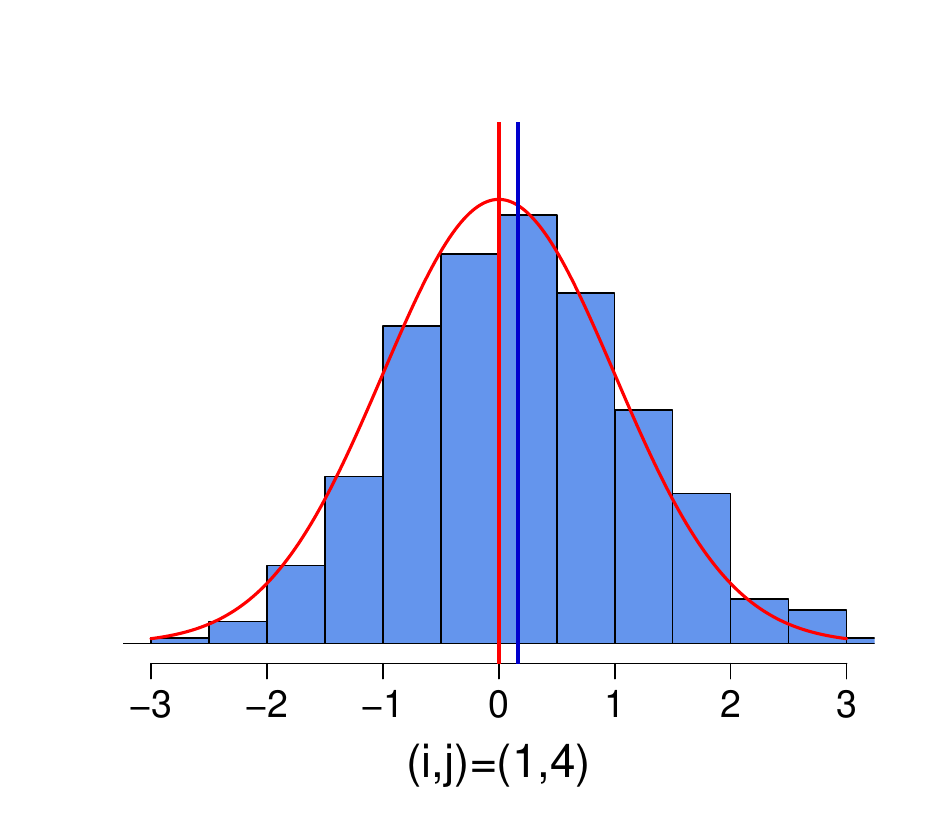}
    \end{minipage}
    \begin{minipage}{0.24\linewidth}
        \centering
        \includegraphics[width=\textwidth]{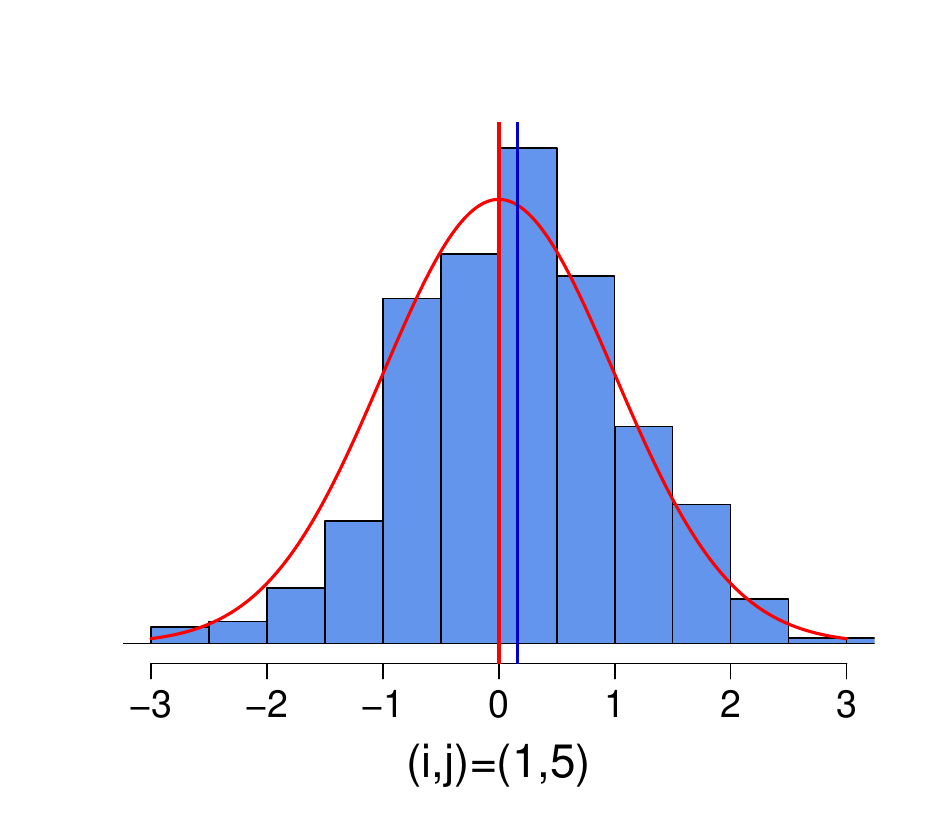}
    \end{minipage}
    \begin{minipage}{0.24\linewidth}
        \centering
        \includegraphics[width=\textwidth]{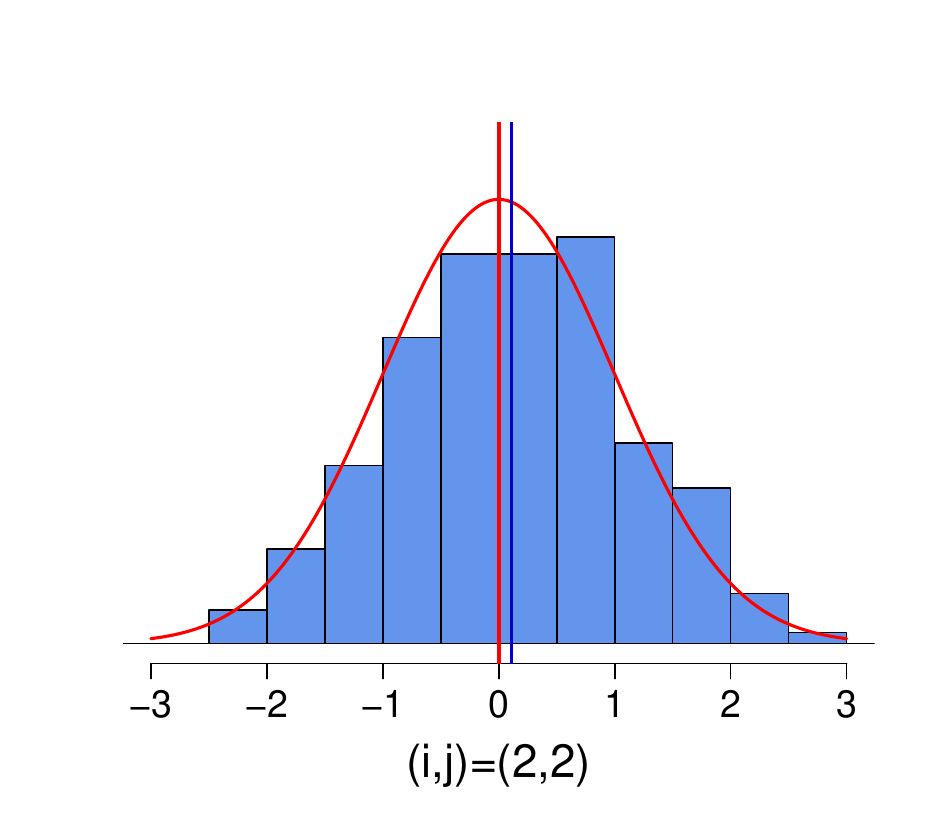}
    \end{minipage}
        \caption*{(b)~~$L_0{:}~ \widehat{\mb{T}}$}
 \end{minipage}   
      \hspace{1cm}
 \begin{minipage}{0.3\linewidth}
    \begin{minipage}{0.24\linewidth}
        \centering
        \includegraphics[width=\textwidth]{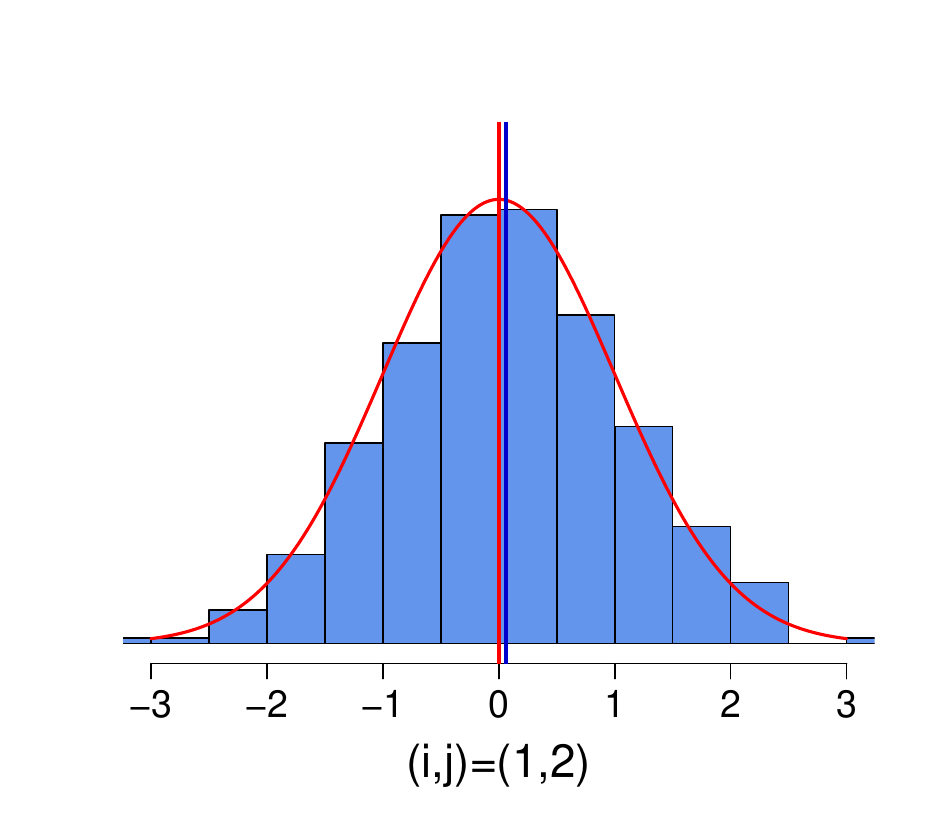}
    \end{minipage}
    \begin{minipage}{0.24\linewidth}
        \centering
        \includegraphics[width=\textwidth]{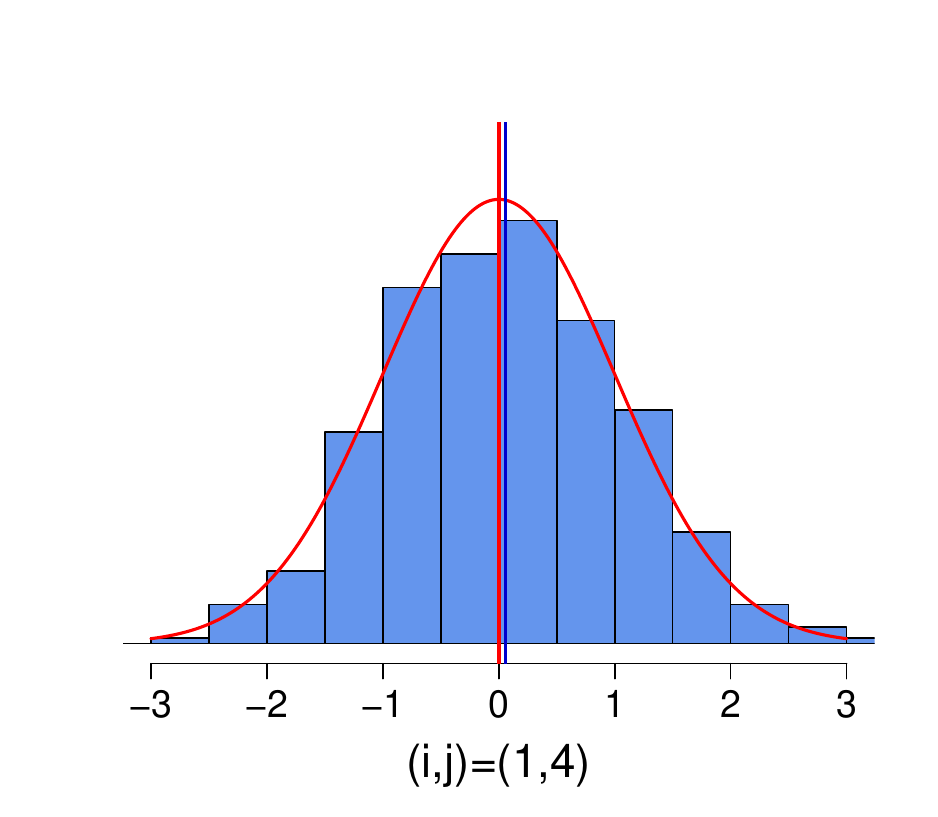}
    \end{minipage}
    \begin{minipage}{0.24\linewidth}
        \centering
        \includegraphics[width=\textwidth]{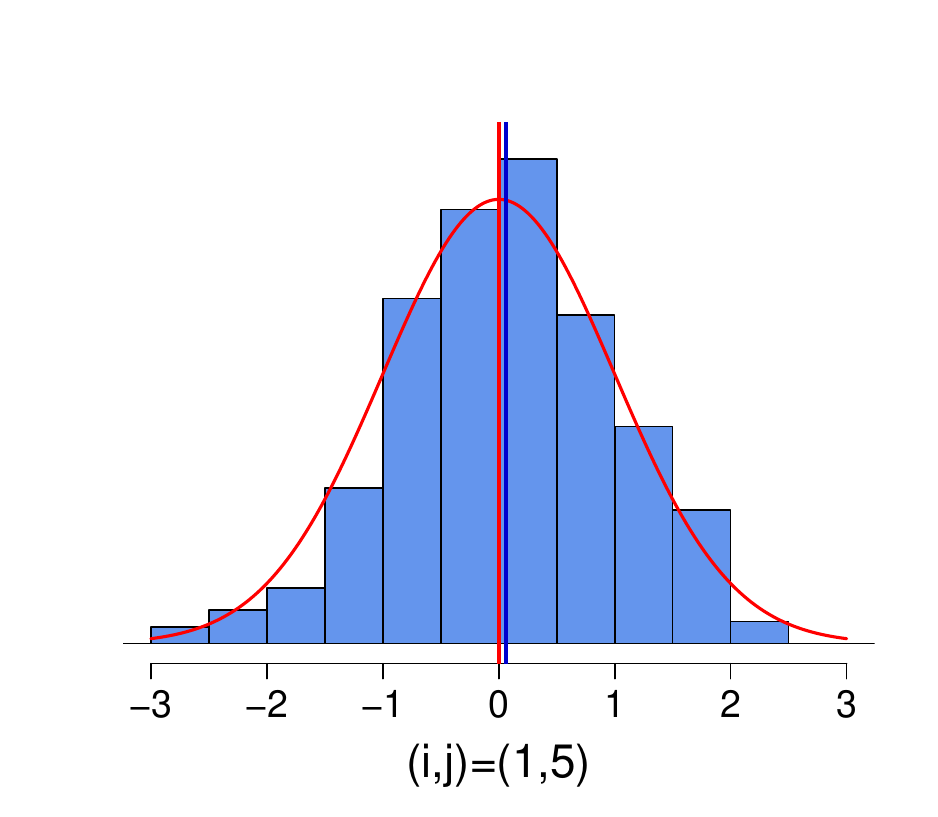}
    \end{minipage}
    \begin{minipage}{0.24\linewidth}
        \centering
        \includegraphics[width=\textwidth]{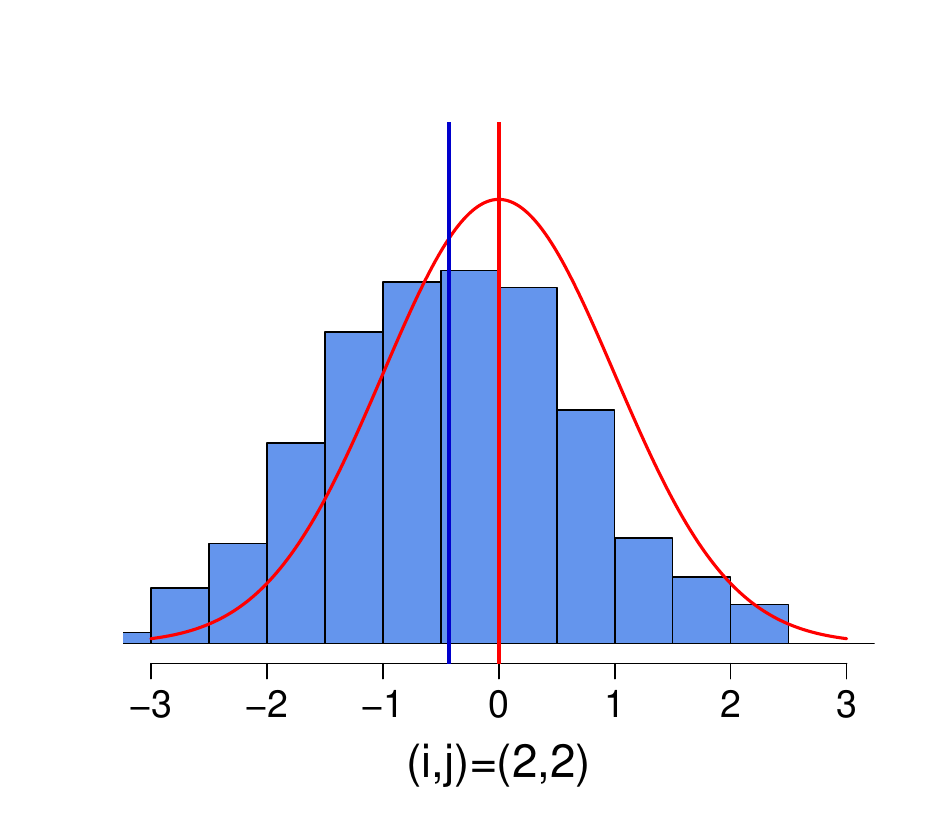}
    \end{minipage}
    \caption*{(c)~~$L_1{:}~ \widehat{\mb{T}}$}
     \end{minipage}   
     \caption{Histograms of $\big(\sqrt{n}(\widehat{\mb{\Omega}}_{ij}^{(m)}-\mb{\Omega}_{ij})/\widehat{\sigma}_{\mb{\Omega}_{ij}}^{(m)}\big)_{m=1}^{400}$ under Gaussian hub graph settings.}
\end{sidewaysfigure}

%Gaussian cluster
 \begin{sidewaysfigure}[th!]
  \caption*{$n=200, p=200$}
      \vspace{-0.43cm}
 \begin{minipage}{0.3\linewidth}
    \begin{minipage}{0.24\linewidth}
        \centering
        \includegraphics[width=\textwidth]{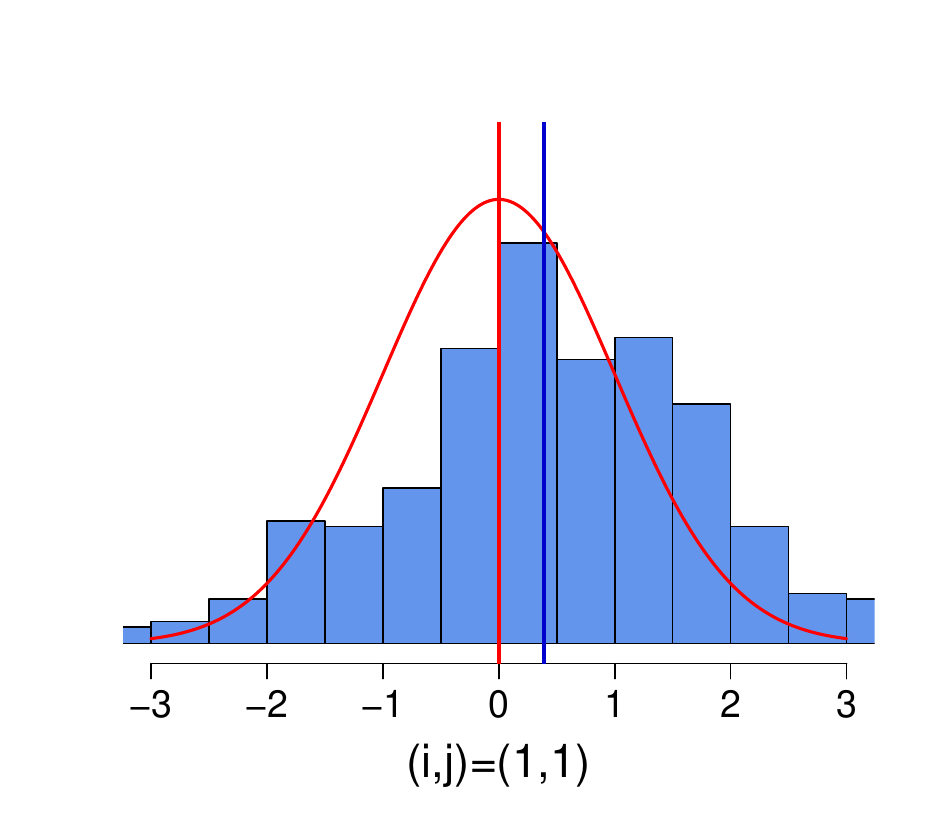}
    \end{minipage}
    \begin{minipage}{0.24\linewidth}
        \centering
        \includegraphics[width=\textwidth]{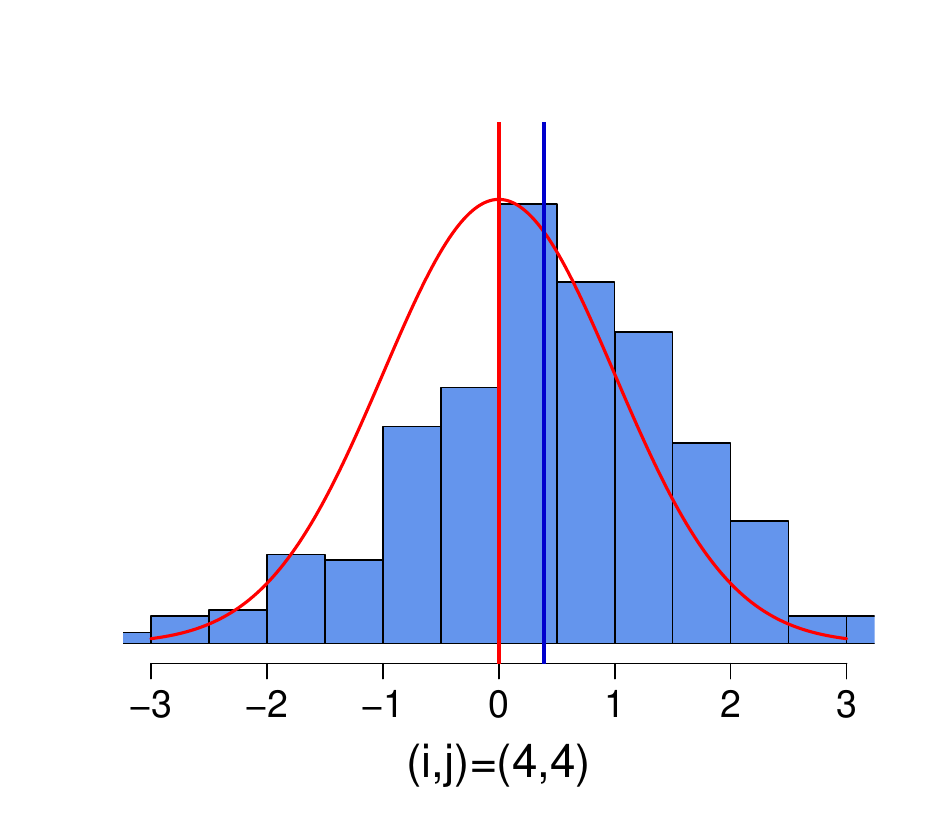}
    \end{minipage}
    \begin{minipage}{0.24\linewidth}
        \centering
        \includegraphics[width=\textwidth]{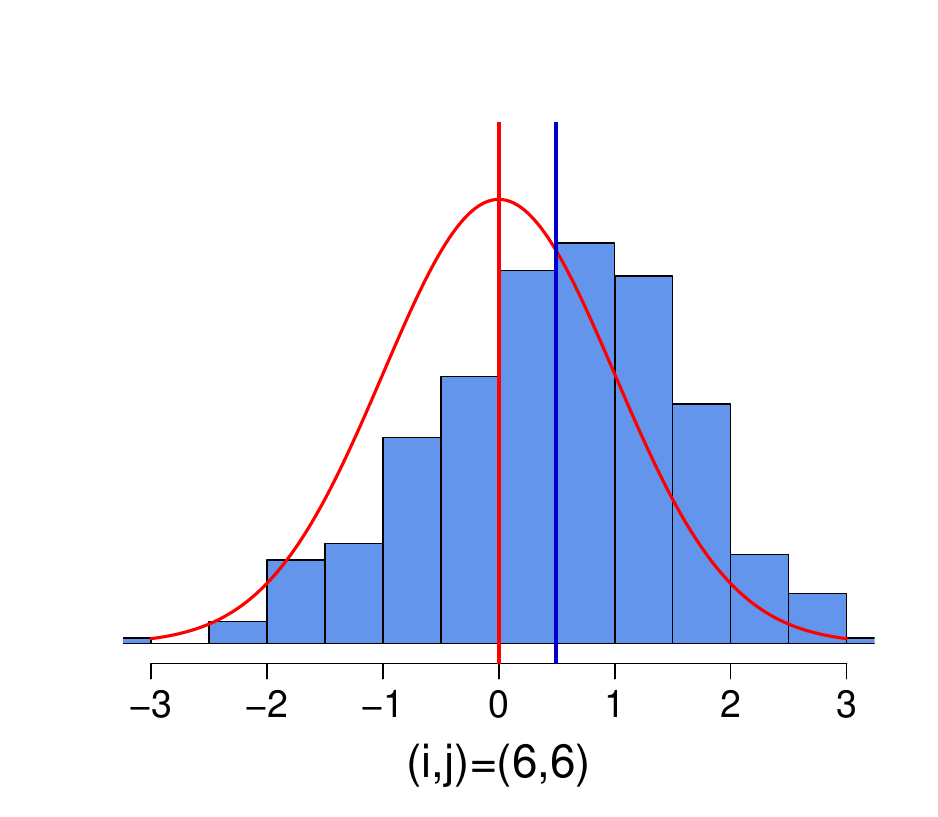}
    \end{minipage}
    \begin{minipage}{0.24\linewidth}
        \centering
        \includegraphics[width=\textwidth]{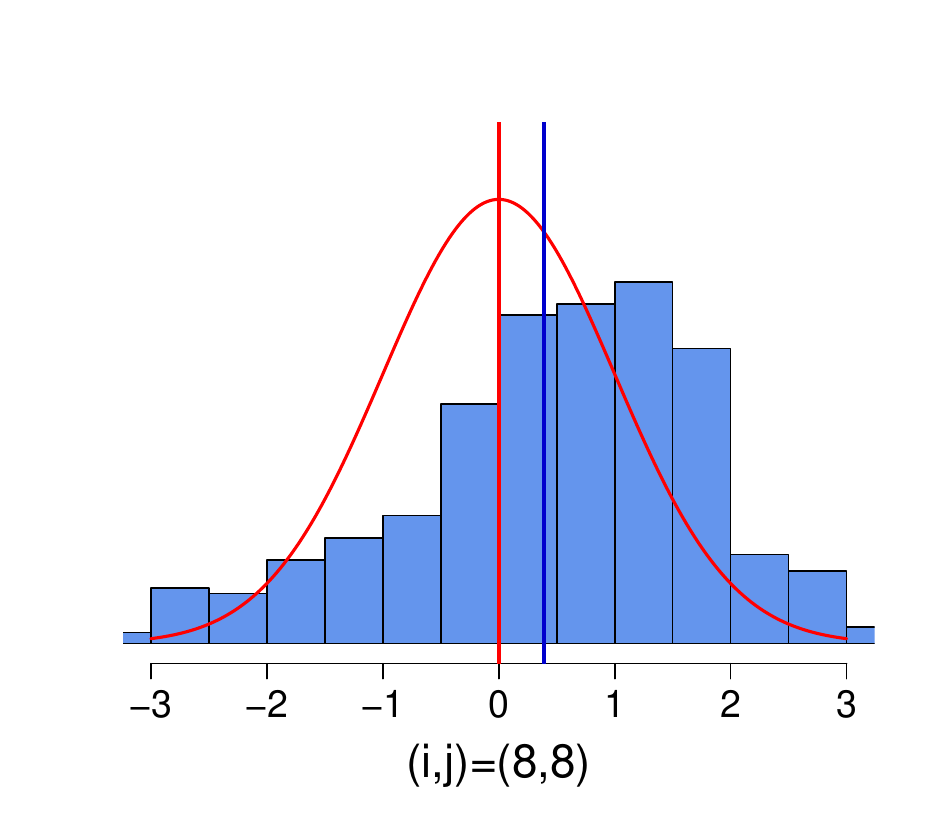}
    \end{minipage}
 \end{minipage}
 \hspace{1cm}
 \begin{minipage}{0.3\linewidth}
    \begin{minipage}{0.24\linewidth}
        \centering
        \includegraphics[width=\textwidth]{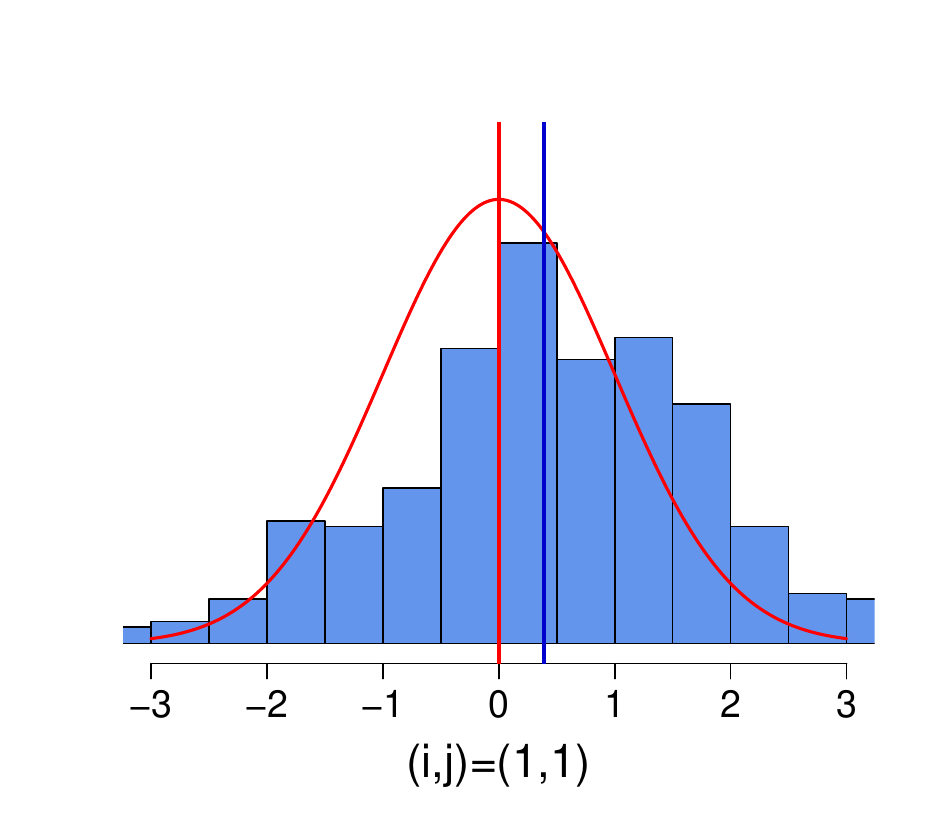}
    \end{minipage}
    \begin{minipage}{0.24\linewidth}
        \centering
        \includegraphics[width=\textwidth]{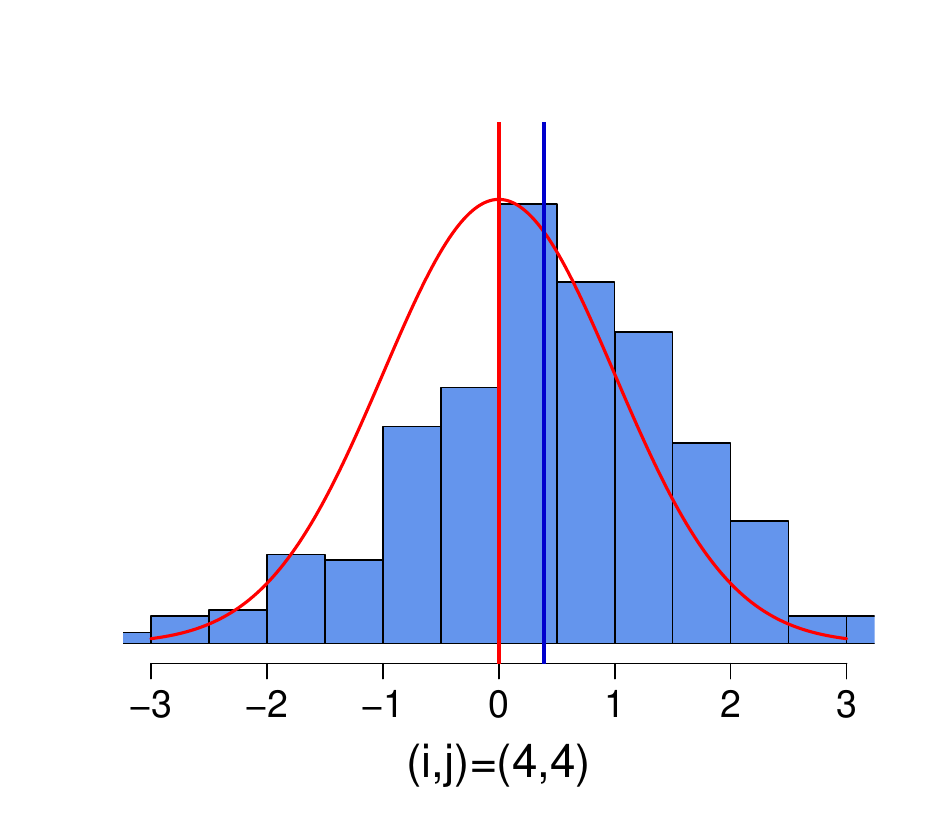}
    \end{minipage}
    \begin{minipage}{0.24\linewidth}
        \centering
        \includegraphics[width=\textwidth]{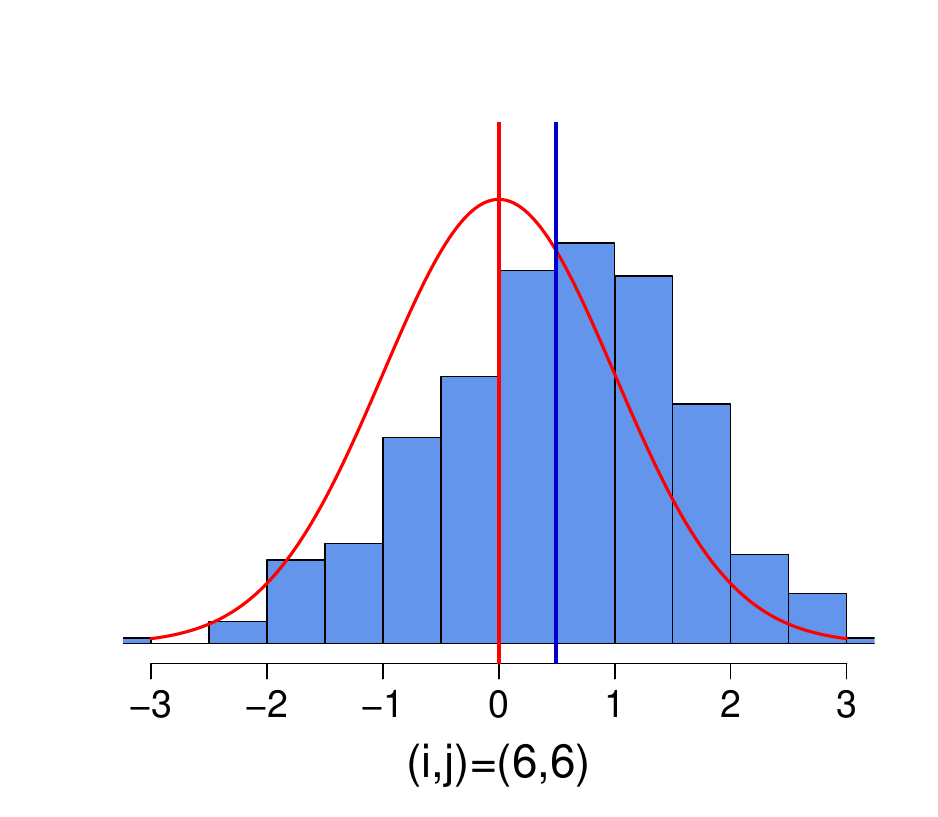}
    \end{minipage}
    \begin{minipage}{0.24\linewidth}
        \centering
        \includegraphics[width=\textwidth]{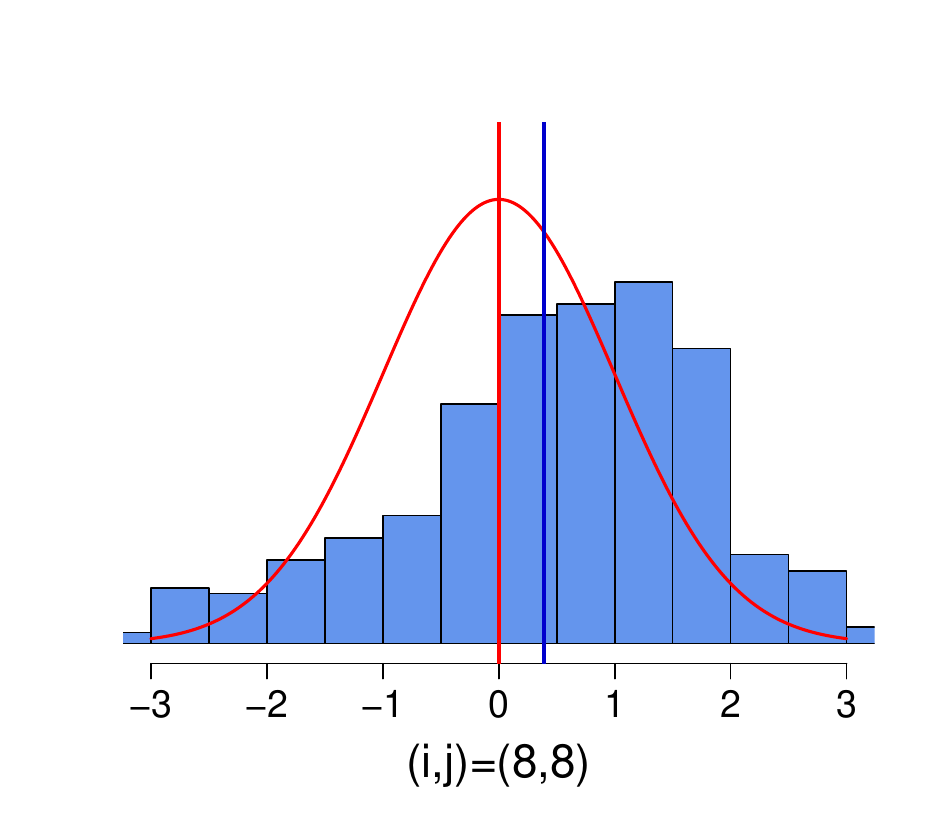}
    \end{minipage}    
 \end{minipage}
  \hspace{1cm}
 \begin{minipage}{0.3\linewidth}
     \begin{minipage}{0.24\linewidth}
        \centering
        \includegraphics[width=\textwidth]{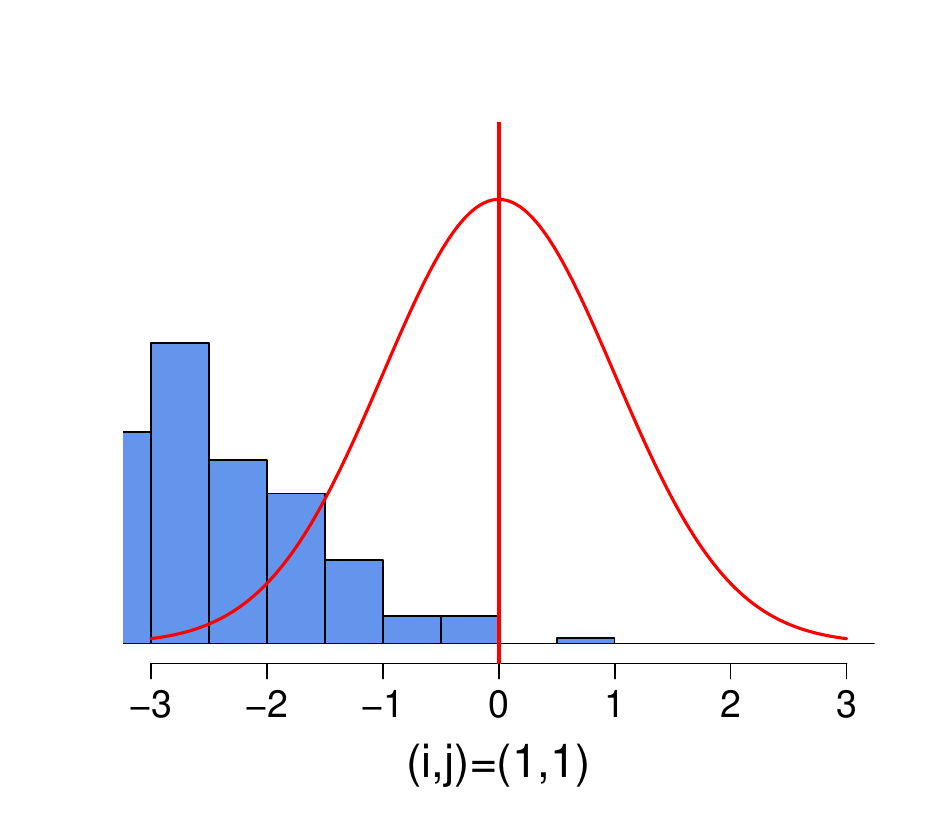}
    \end{minipage}
    \begin{minipage}{0.24\linewidth}
        \centering
        \includegraphics[width=\textwidth]{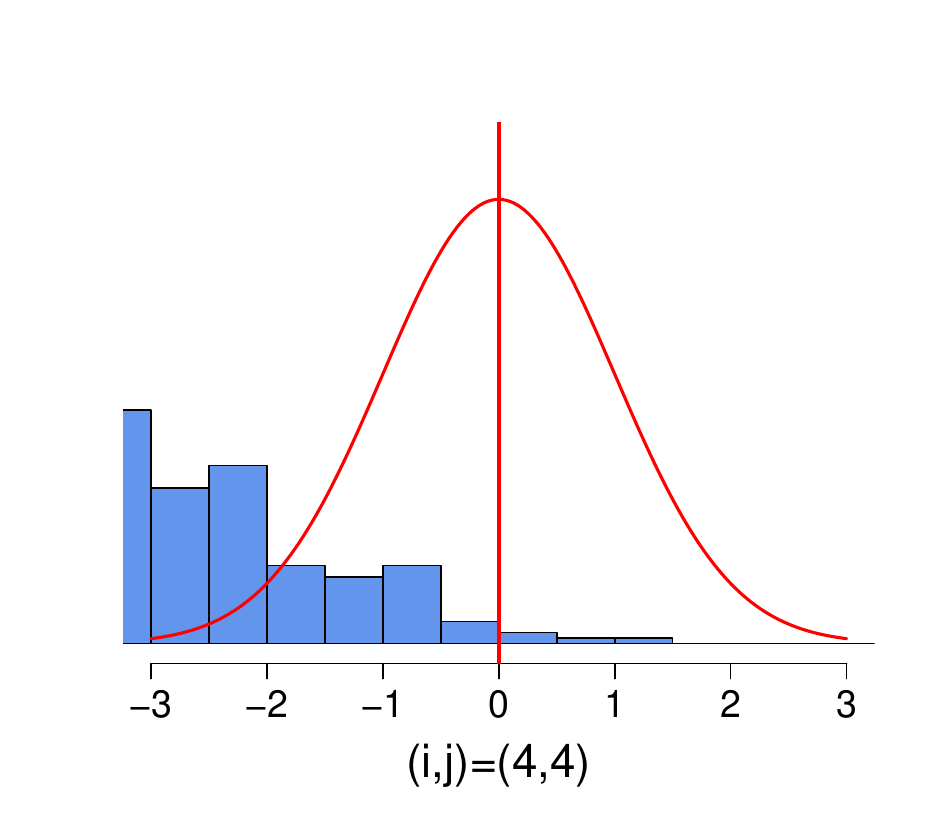}
    \end{minipage}
    \begin{minipage}{0.24\linewidth}
        \centering
        \includegraphics[width=\textwidth]{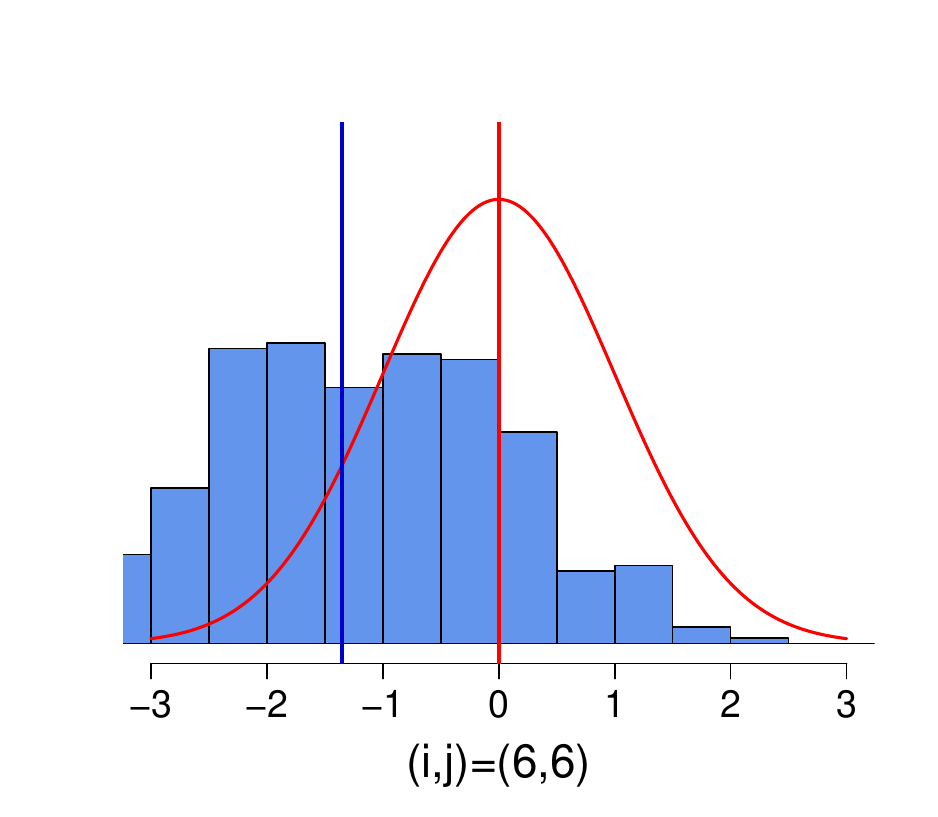}
    \end{minipage}
    \begin{minipage}{0.24\linewidth}
        \centering
        \includegraphics[width=\textwidth]{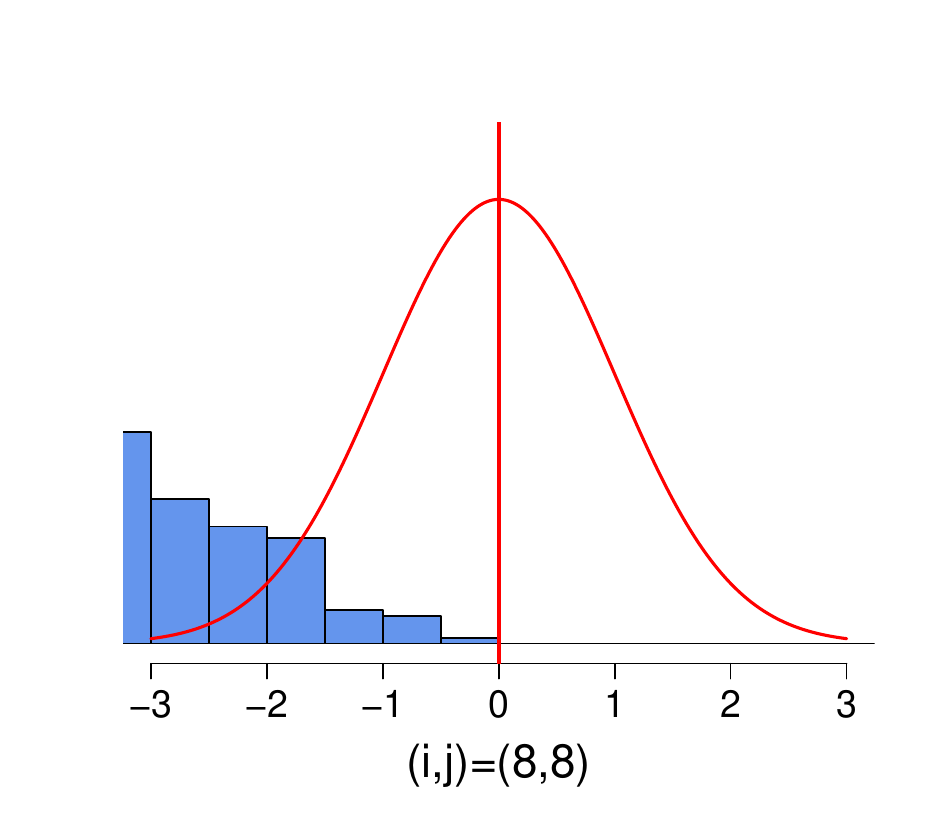}
    \end{minipage}
 \end{minipage}

  \caption*{$n=400, p=200$}
      \vspace{-0.43cm}
 \begin{minipage}{0.3\linewidth}
    \begin{minipage}{0.24\linewidth}
        \centering
        \includegraphics[width=\textwidth]{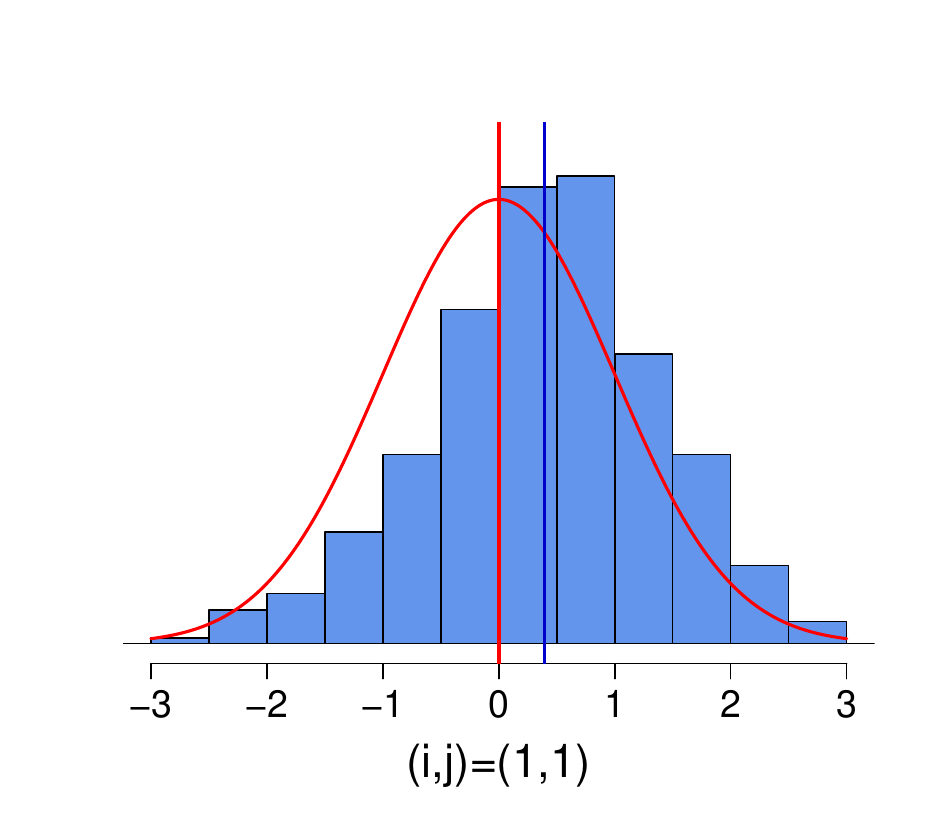}
    \end{minipage}
    \begin{minipage}{0.24\linewidth}
        \centering
        \includegraphics[width=\textwidth]{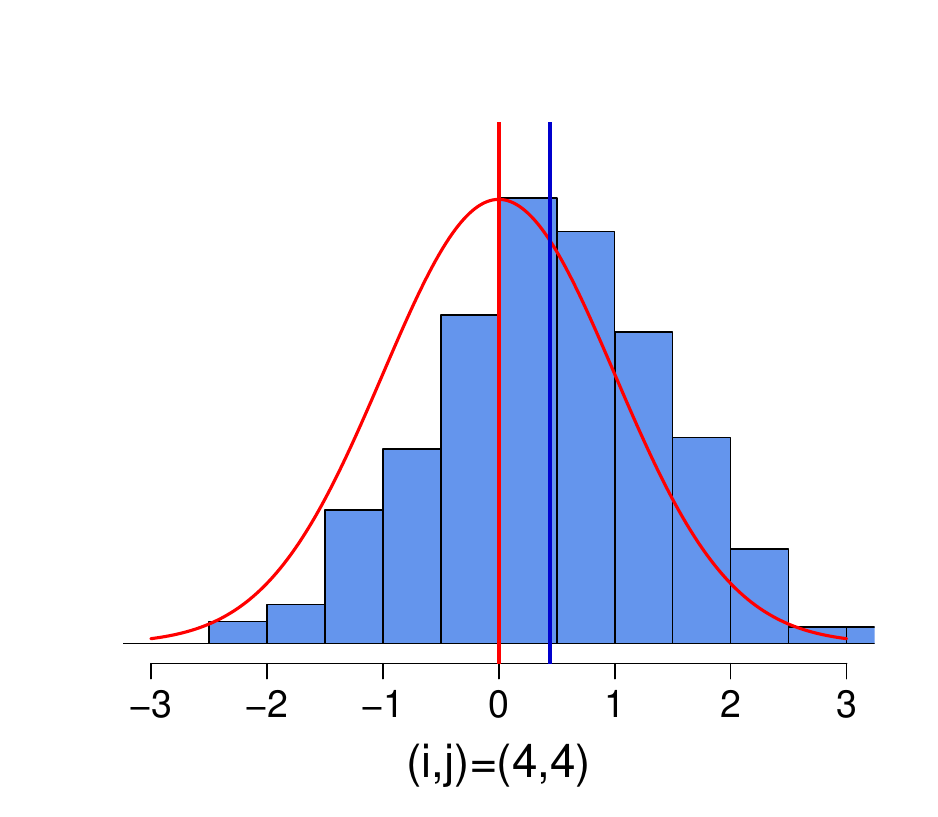}
    \end{minipage}
    \begin{minipage}{0.24\linewidth}
        \centering
        \includegraphics[width=\textwidth]{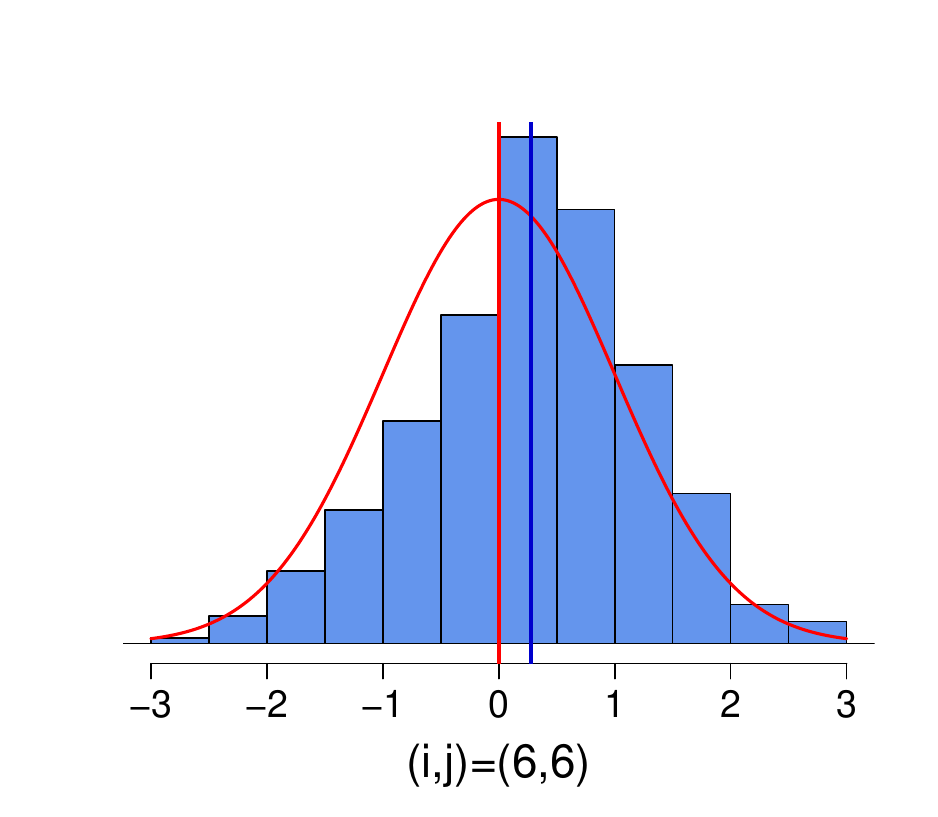}
    \end{minipage}
    \begin{minipage}{0.24\linewidth}
        \centering
        \includegraphics[width=\textwidth]{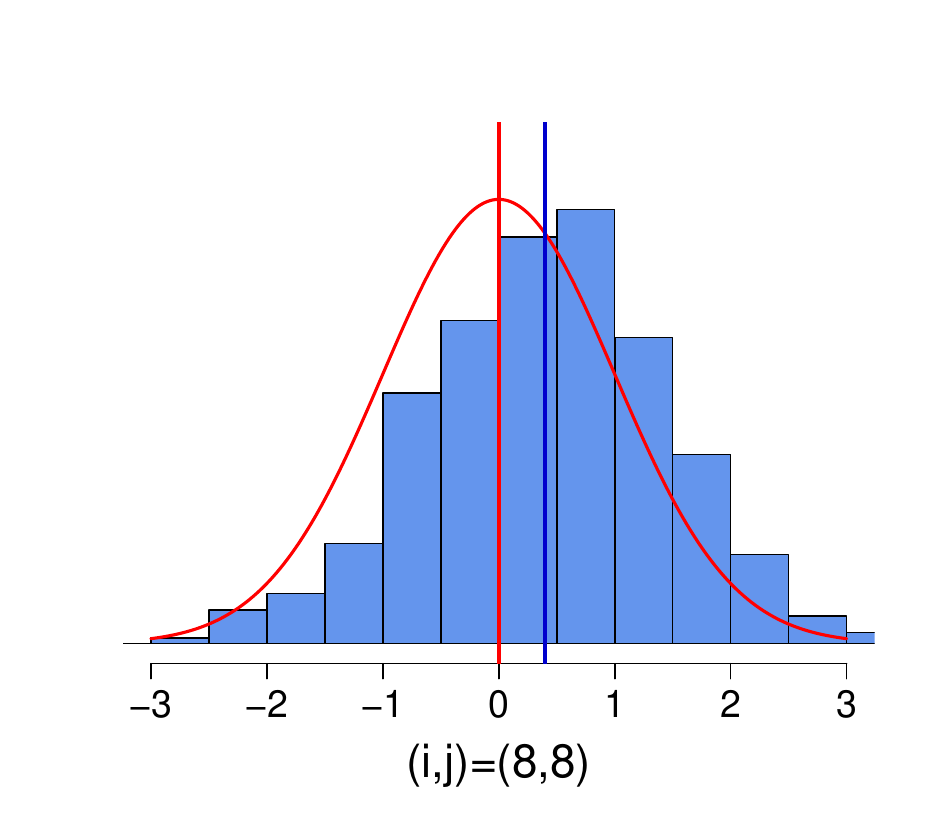}
    \end{minipage}
 \end{minipage}  
     \hspace{1cm}
 \begin{minipage}{0.3\linewidth}
    \begin{minipage}{0.24\linewidth}
        \centering
        \includegraphics[width=\textwidth]{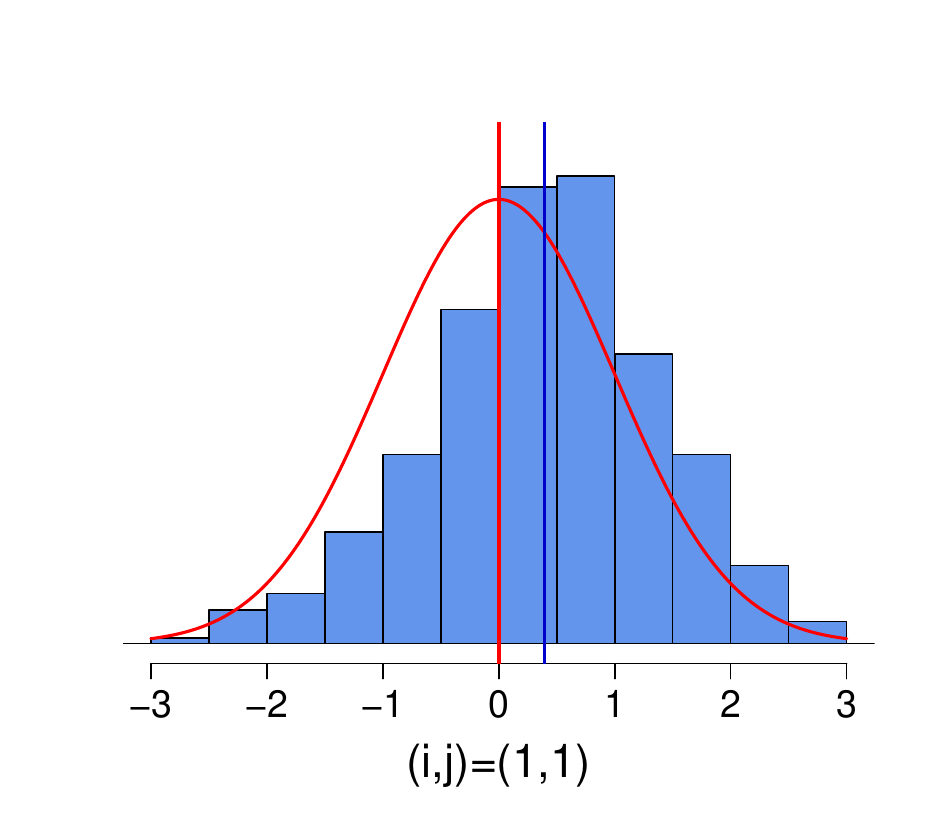}
    \end{minipage}
    \begin{minipage}{0.24\linewidth}
        \centering
        \includegraphics[width=\textwidth]{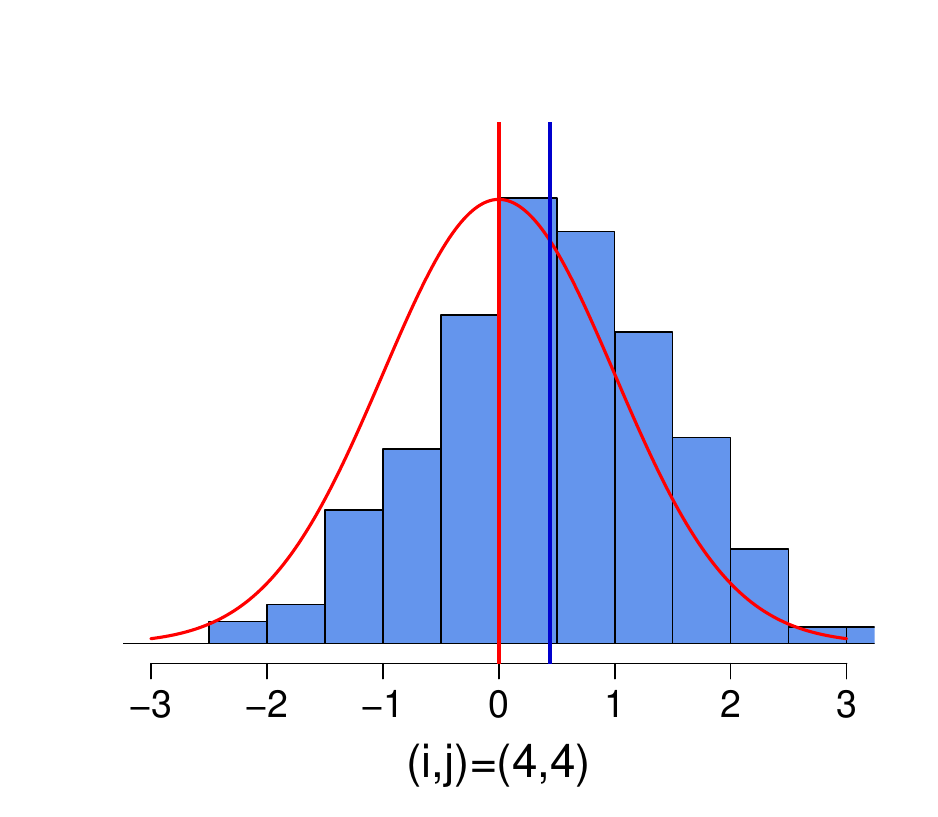}
    \end{minipage}
    \begin{minipage}{0.24\linewidth}
        \centering
        \includegraphics[width=\textwidth]{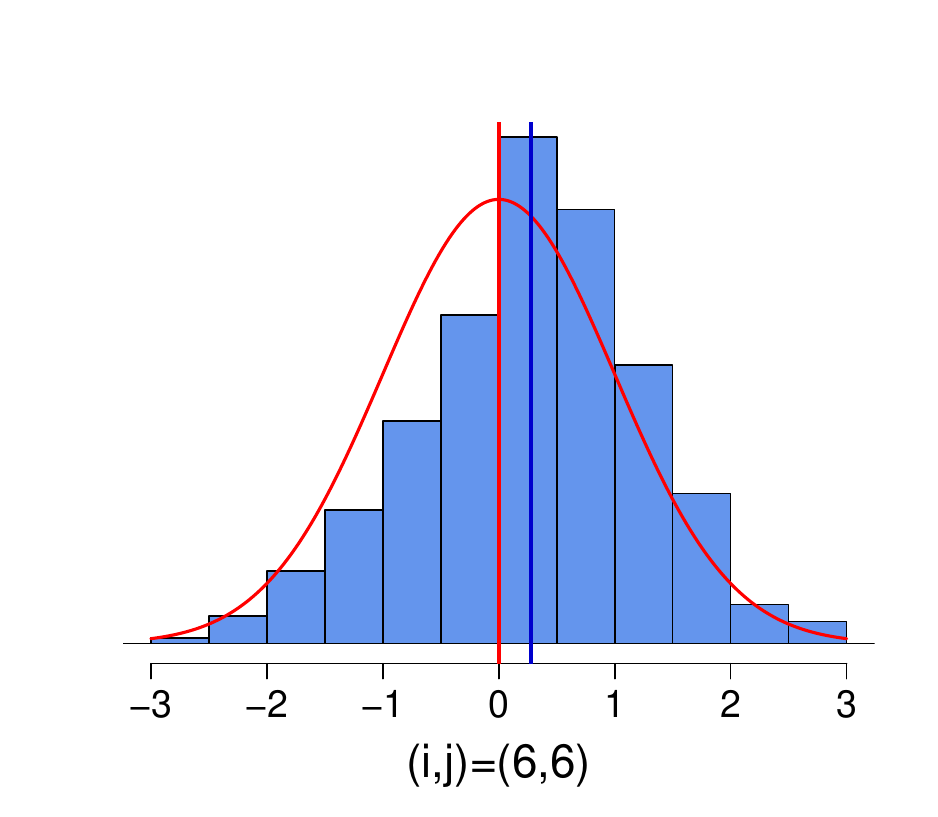}
    \end{minipage}
    \begin{minipage}{0.24\linewidth}
        \centering
        \includegraphics[width=\textwidth]{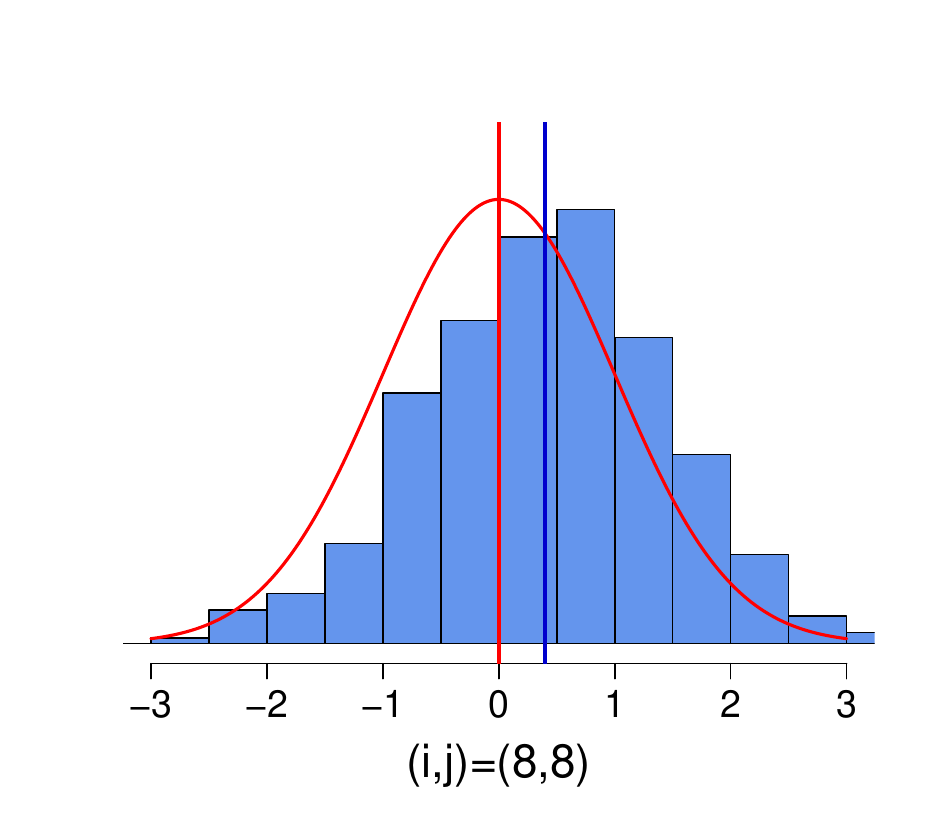}
    \end{minipage}
  \end{minipage}  
    \hspace{1cm}
 \begin{minipage}{0.3\linewidth}
    \begin{minipage}{0.24\linewidth}
        \centering
        \includegraphics[width=\textwidth]{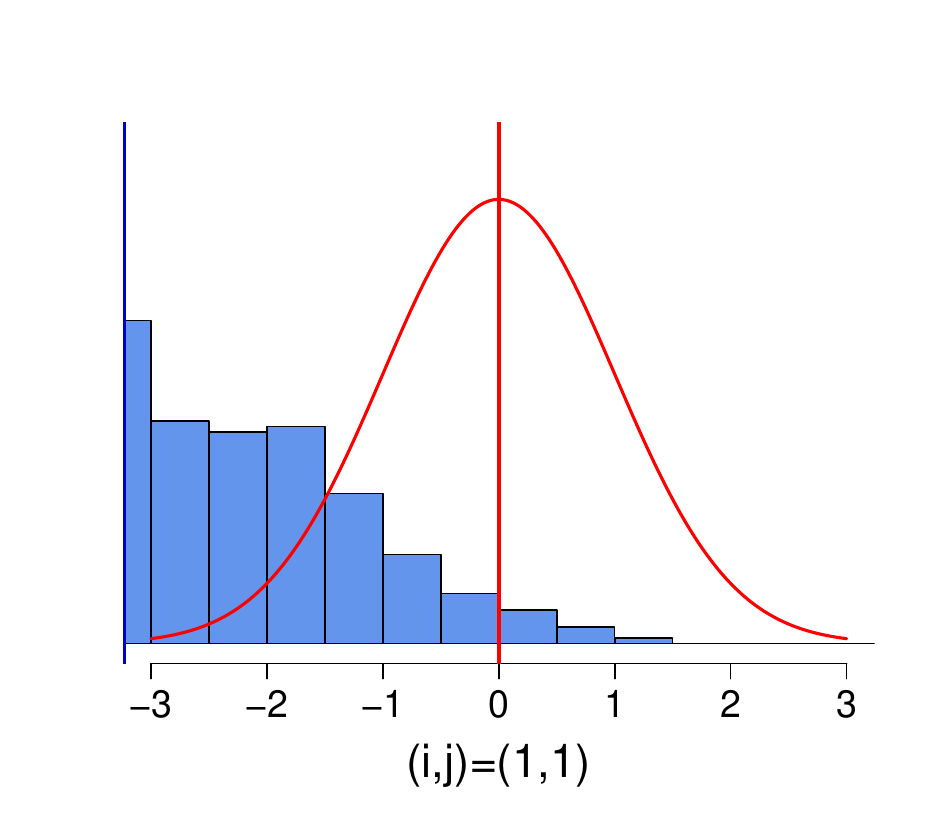}
    \end{minipage}
    \begin{minipage}{0.24\linewidth}
        \centering
        \includegraphics[width=\textwidth]{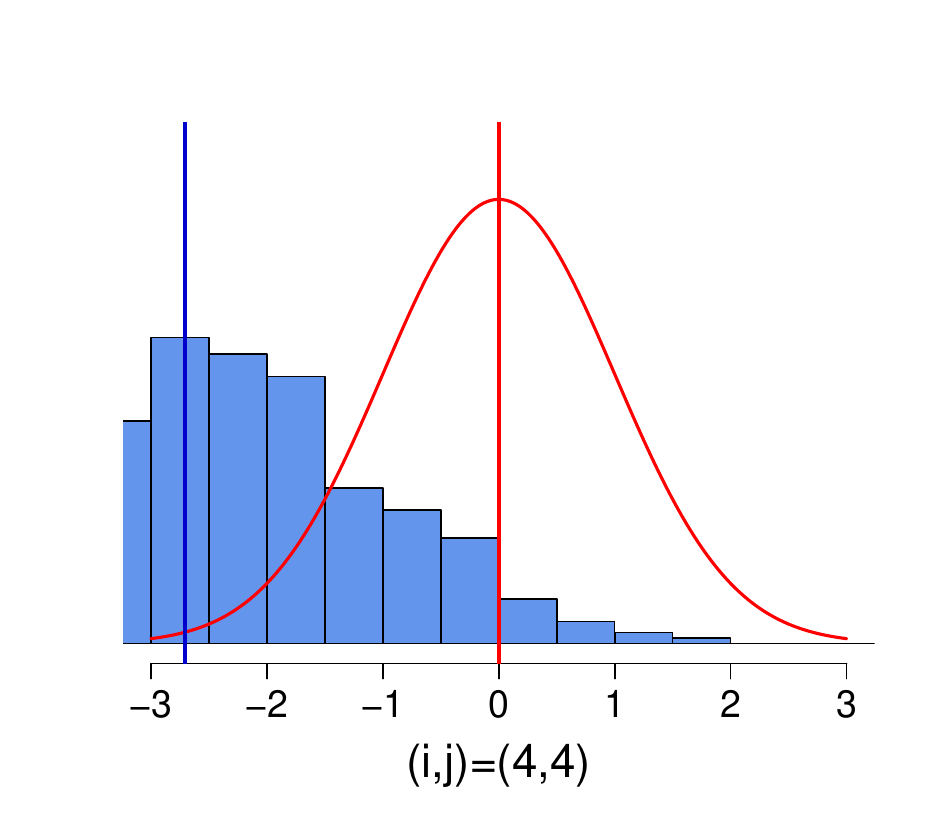}
    \end{minipage}
    \begin{minipage}{0.24\linewidth}
        \centering
        \includegraphics[width=\textwidth]{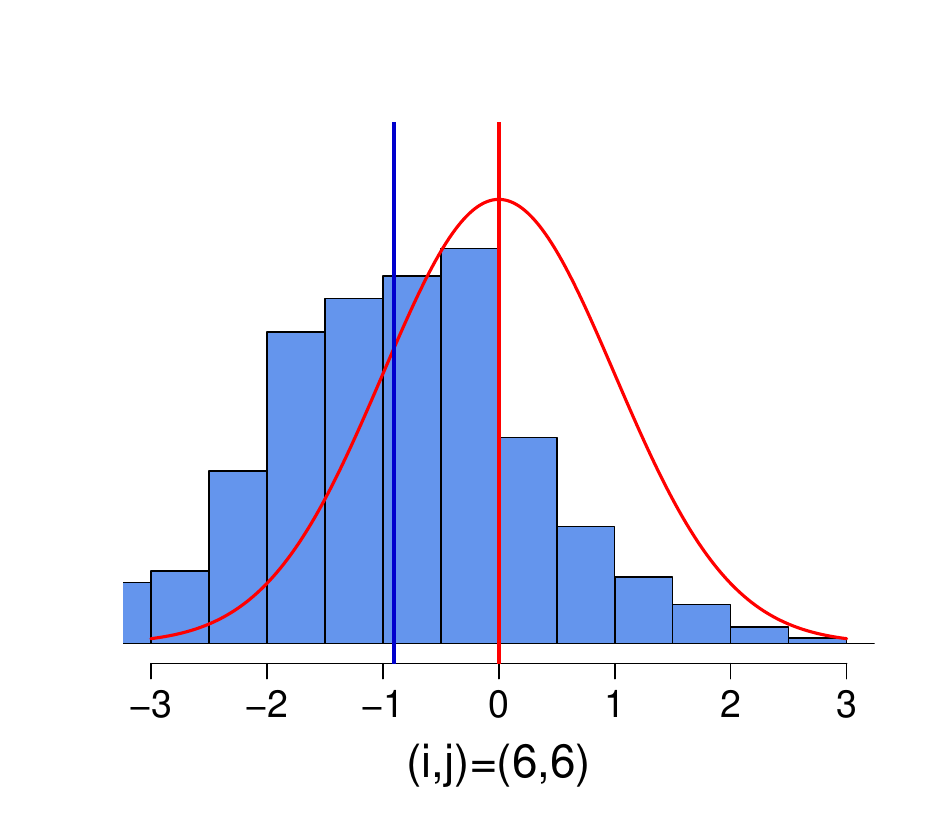}
    \end{minipage}
    \begin{minipage}{0.24\linewidth}
        \centering
        \includegraphics[width=\textwidth]{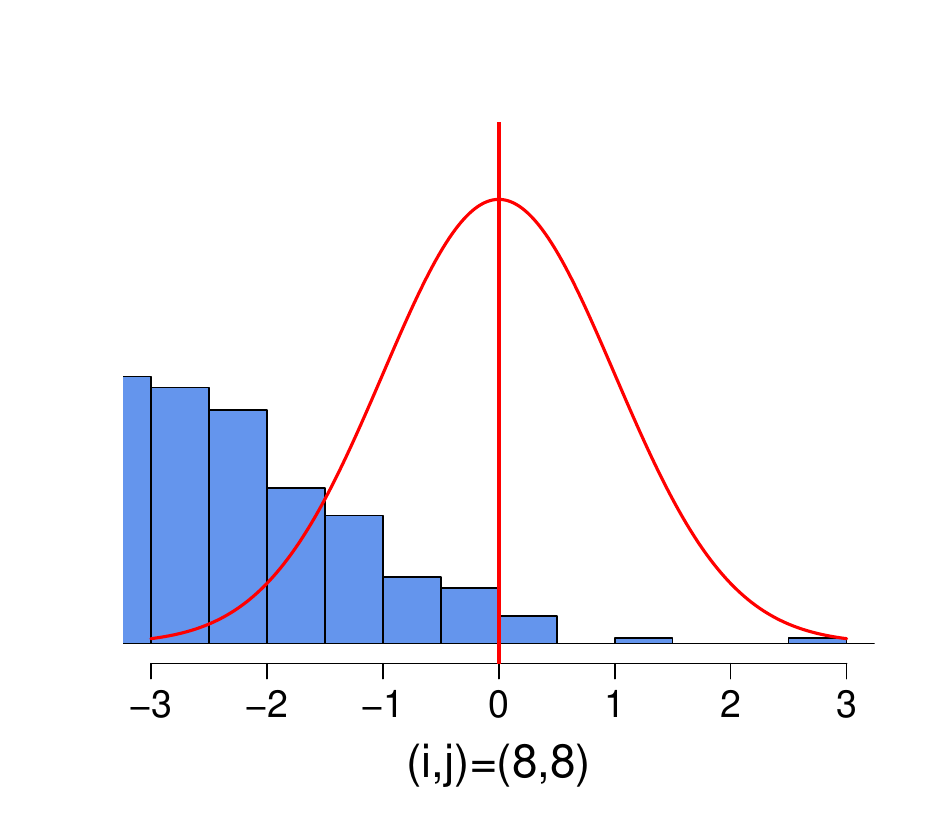}
    \end{minipage}
 \end{minipage}

  \caption*{$n=800, p=200$}
      \vspace{-0.43cm}
 \begin{minipage}{0.3\linewidth}
    \begin{minipage}{0.24\linewidth}
        \centering
        \includegraphics[width=\textwidth]{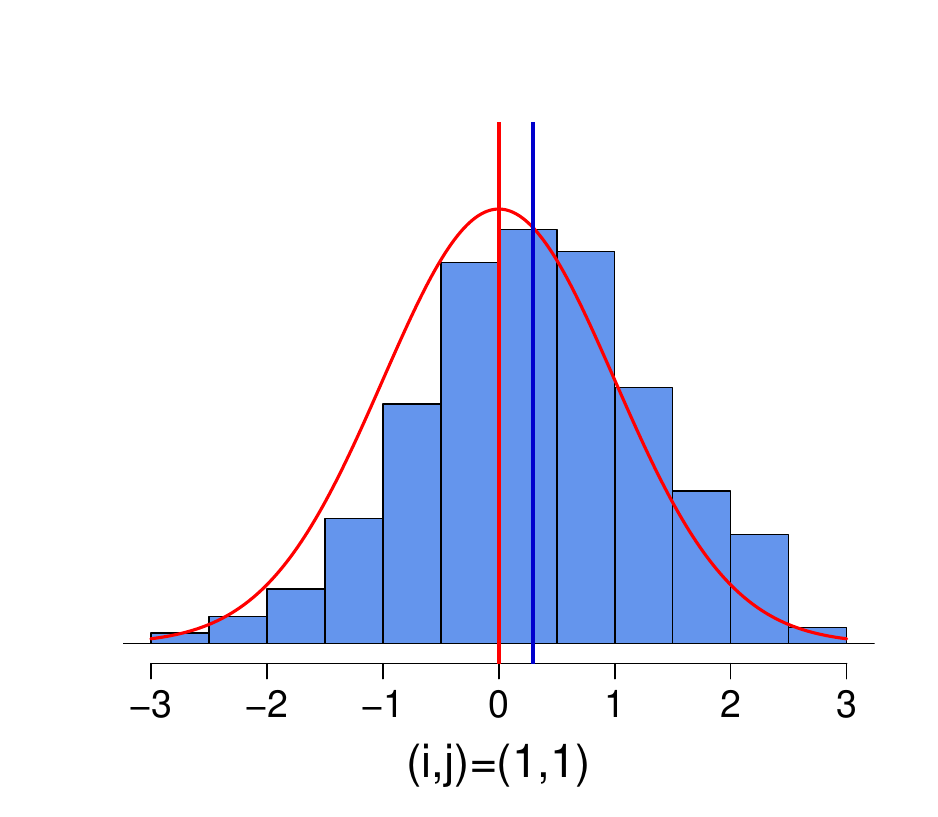}
    \end{minipage}
    \begin{minipage}{0.24\linewidth}
        \centering
        \includegraphics[width=\textwidth]{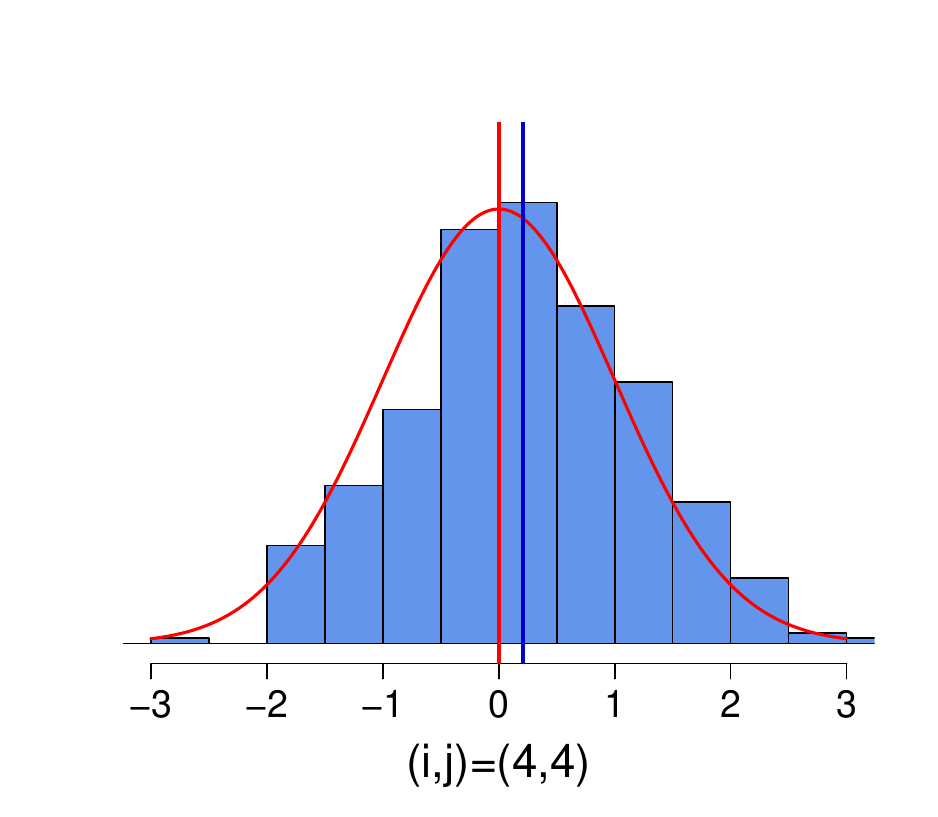}
    \end{minipage}
    \begin{minipage}{0.24\linewidth}
        \centering
        \includegraphics[width=\textwidth]{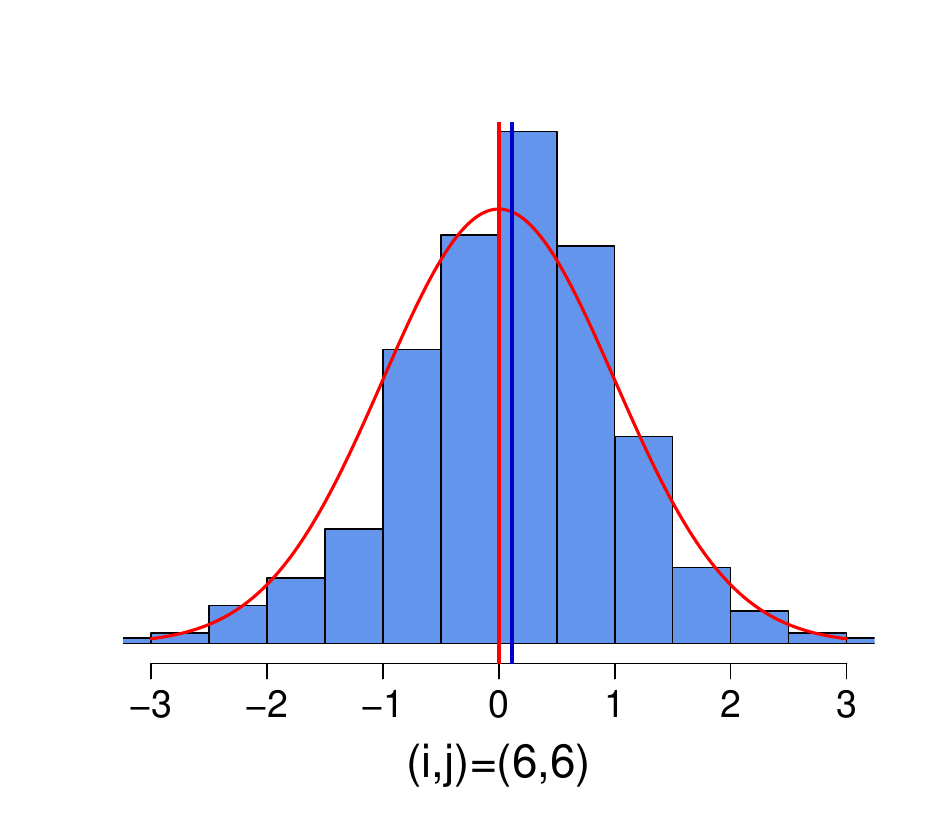}
    \end{minipage}
    \begin{minipage}{0.24\linewidth}
        \centering
        \includegraphics[width=\textwidth]{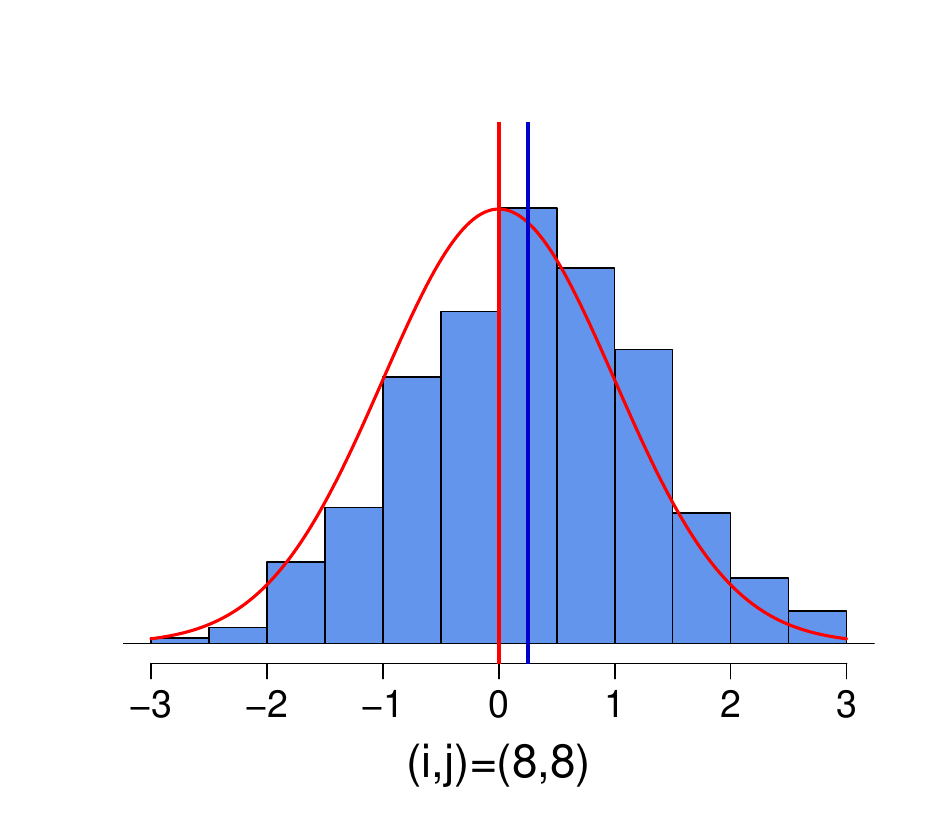}
    \end{minipage}
 \end{minipage} 
     \hspace{1cm}
 \begin{minipage}{0.3\linewidth}
    \begin{minipage}{0.24\linewidth}
        \centering
        \includegraphics[width=\textwidth]{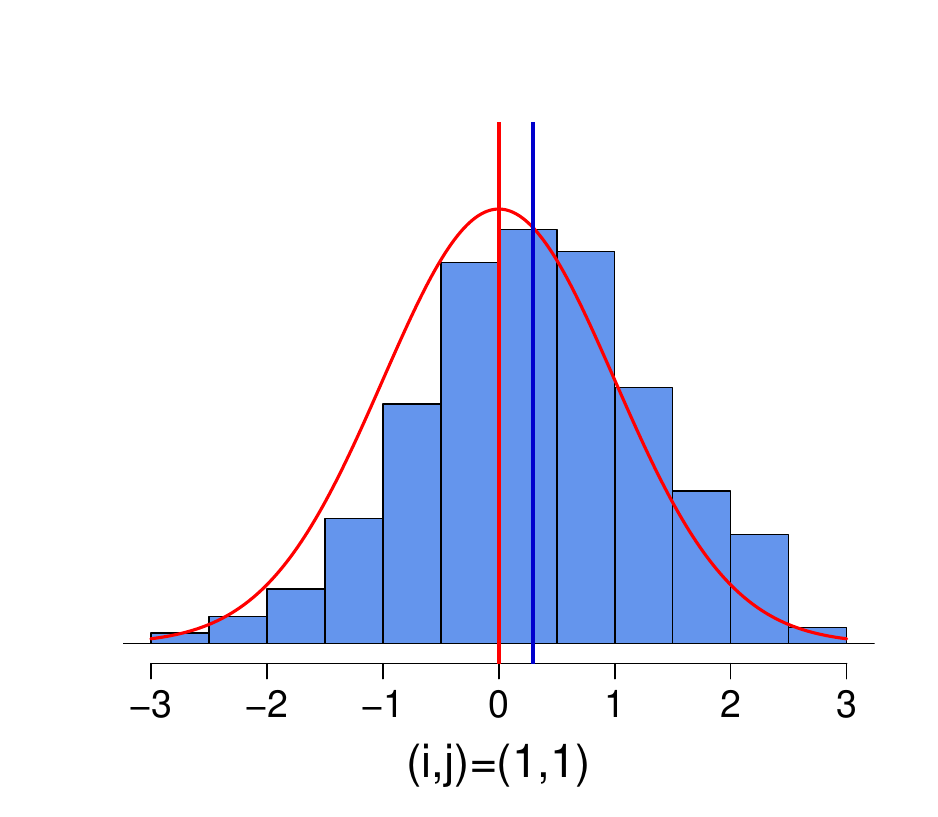}
    \end{minipage}
    \begin{minipage}{0.24\linewidth}
        \centering
        \includegraphics[width=\textwidth]{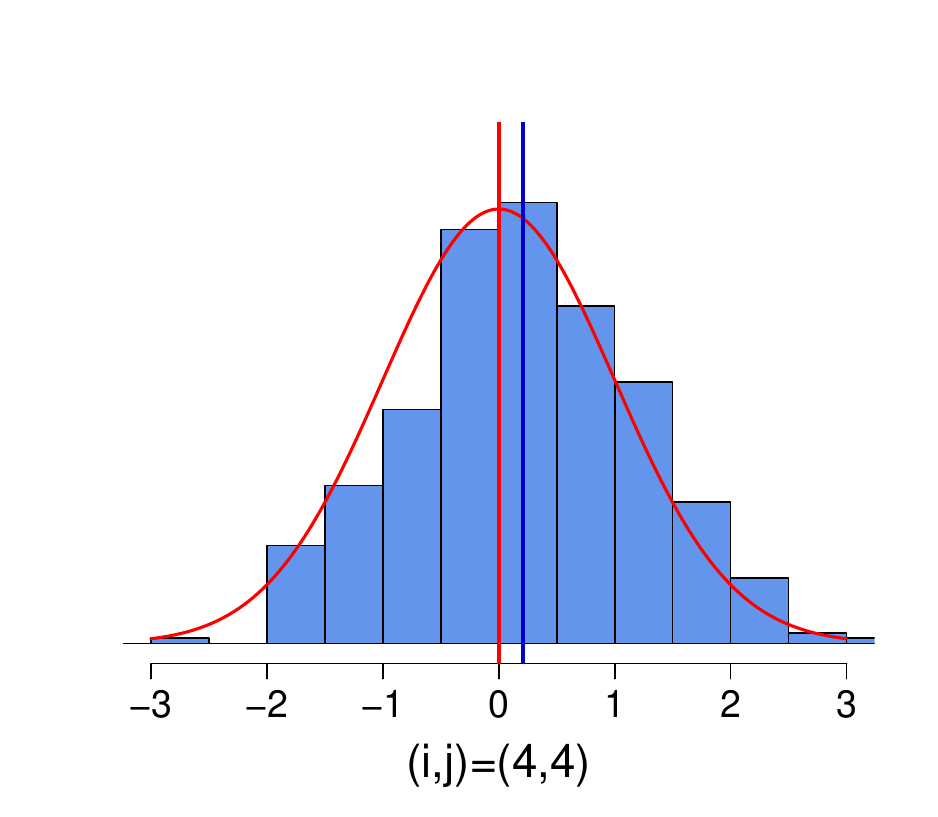}
    \end{minipage}
    \begin{minipage}{0.24\linewidth}
        \centering
        \includegraphics[width=\textwidth]{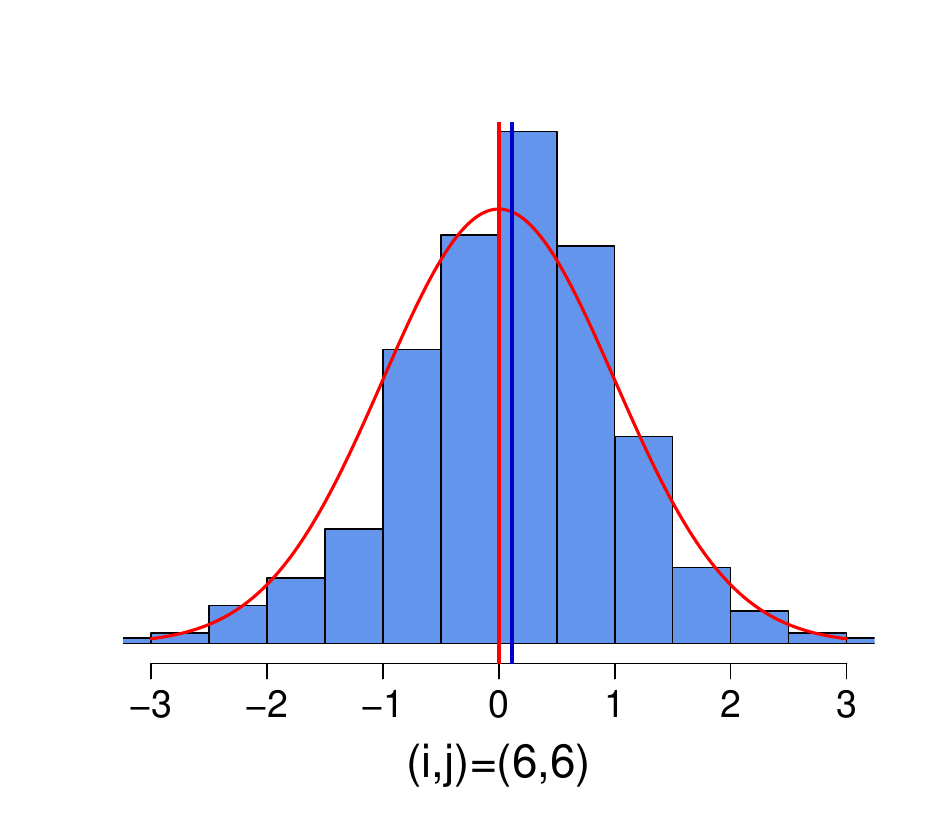}
    \end{minipage}
    \begin{minipage}{0.24\linewidth}
        \centering
        \includegraphics[width=\textwidth]{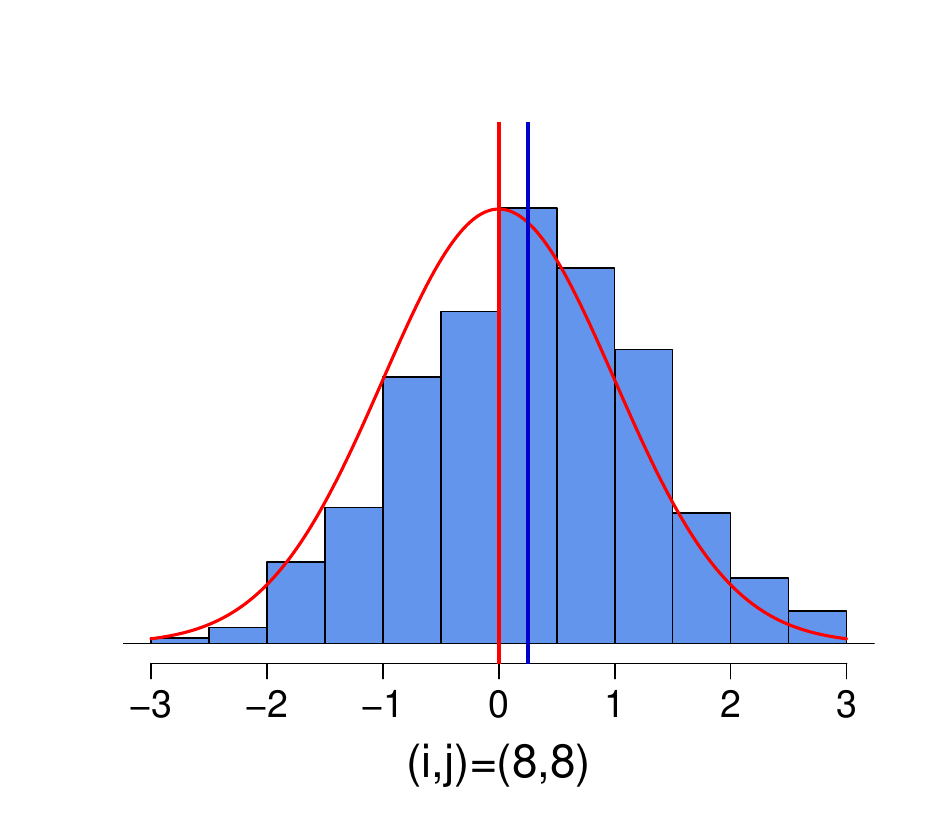}
    \end{minipage}
 \end{minipage}   
      \hspace{1cm}
 \begin{minipage}{0.3\linewidth}
    \begin{minipage}{0.24\linewidth}
        \centering
        \includegraphics[width=\textwidth]{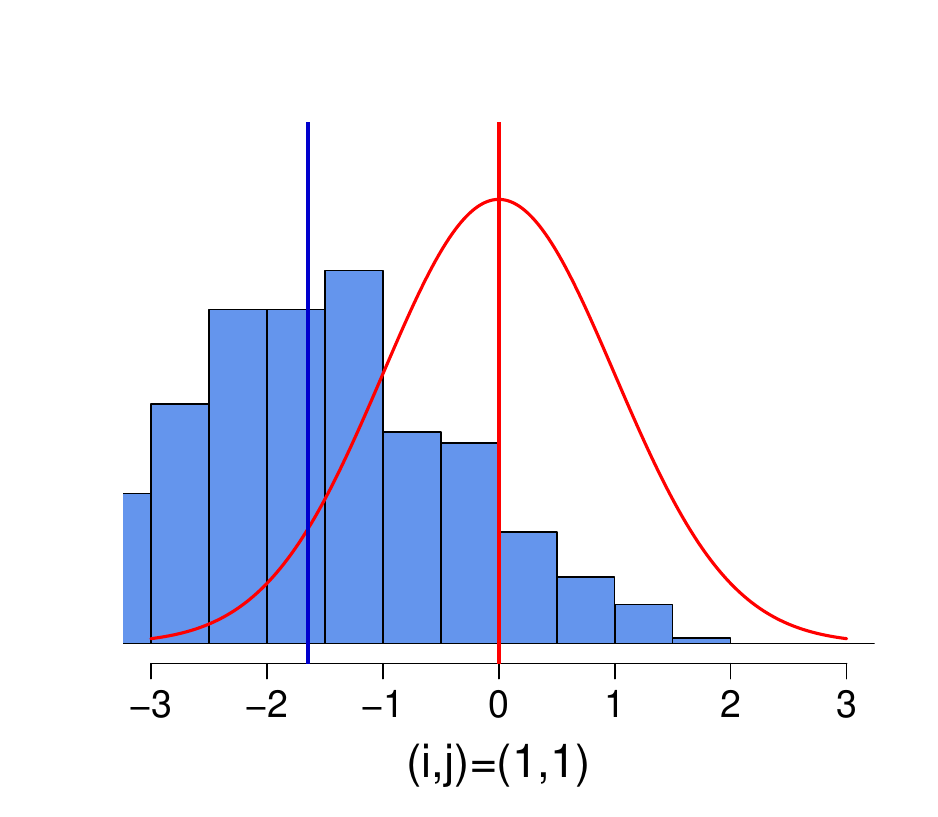}
    \end{minipage}
    \begin{minipage}{0.24\linewidth}
        \centering
        \includegraphics[width=\textwidth]{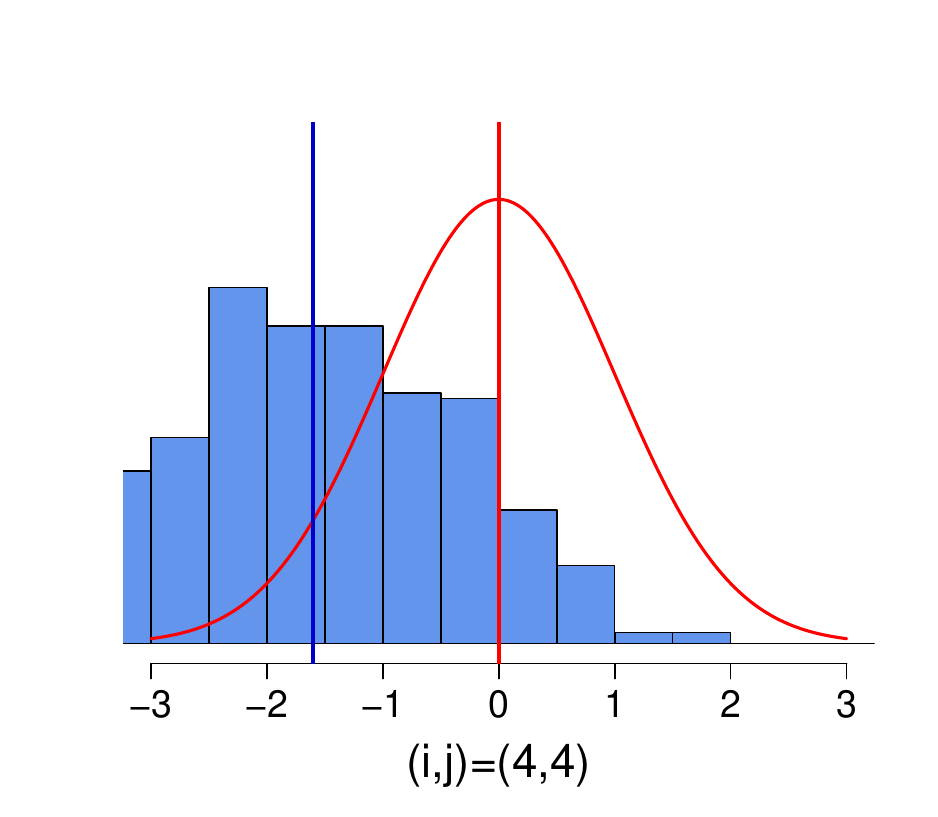}
    \end{minipage}
    \begin{minipage}{0.24\linewidth}
        \centering
        \includegraphics[width=\textwidth]{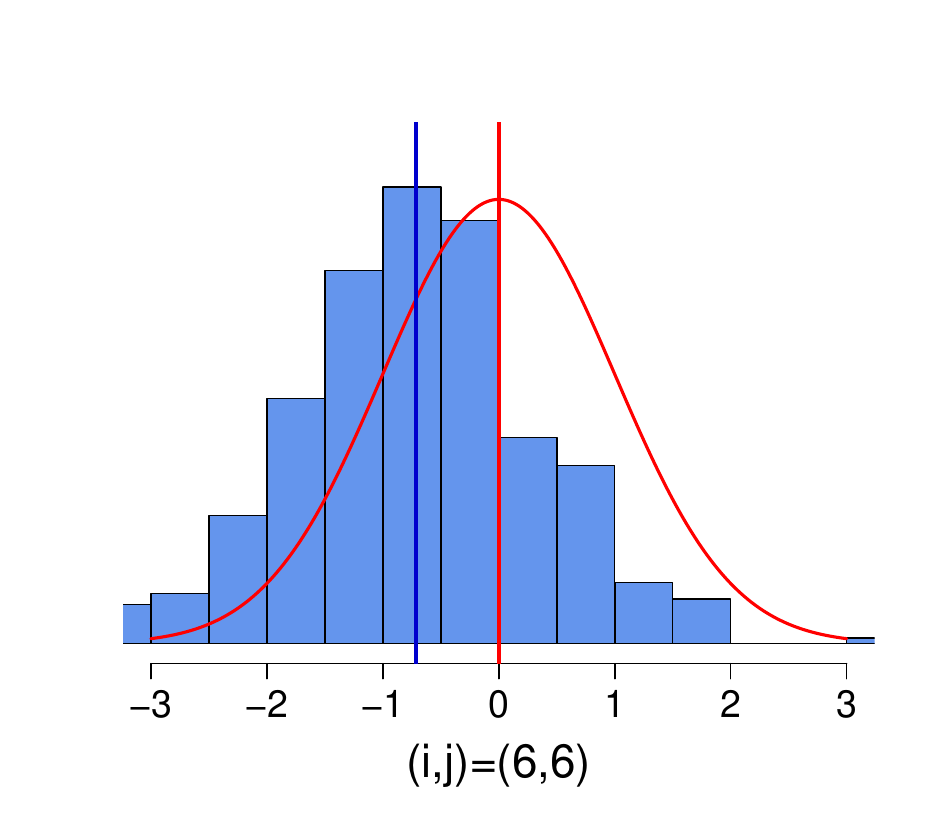}
    \end{minipage}
    \begin{minipage}{0.24\linewidth}
        \centering
        \includegraphics[width=\textwidth]{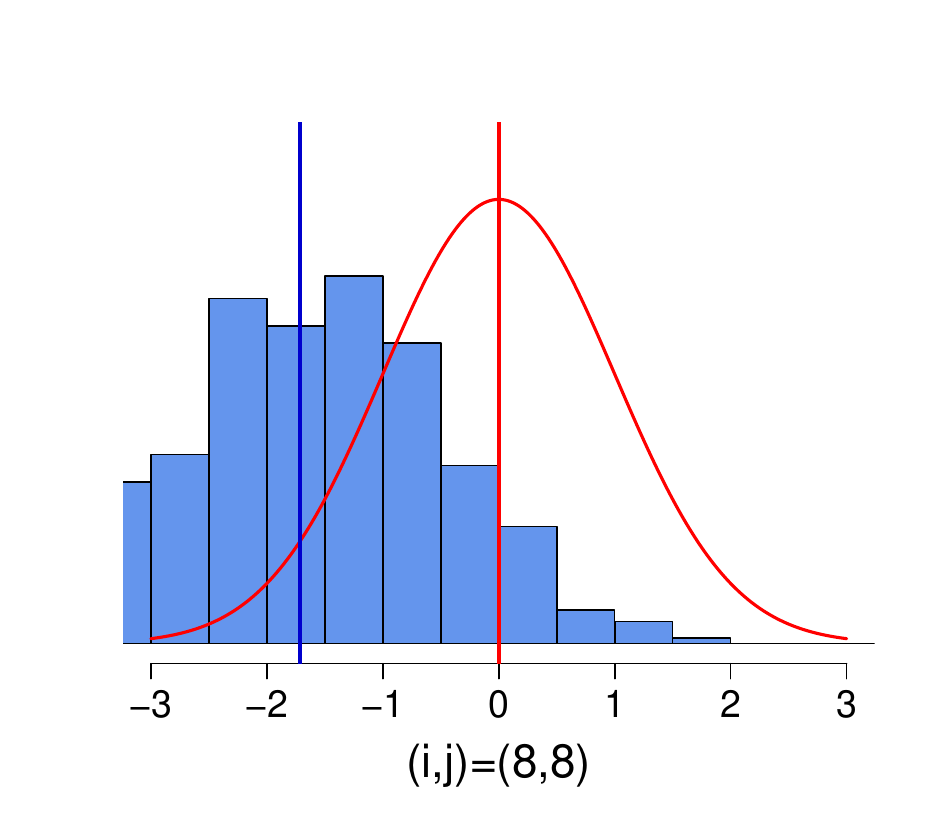}
    \end{minipage}
     \end{minipage}   
     
 \caption*{$n=200, p=400$}
     \vspace{-0.43cm}
 \begin{minipage}{0.3\linewidth}
    \begin{minipage}{0.24\linewidth}
        \centering
        \includegraphics[width=\textwidth]{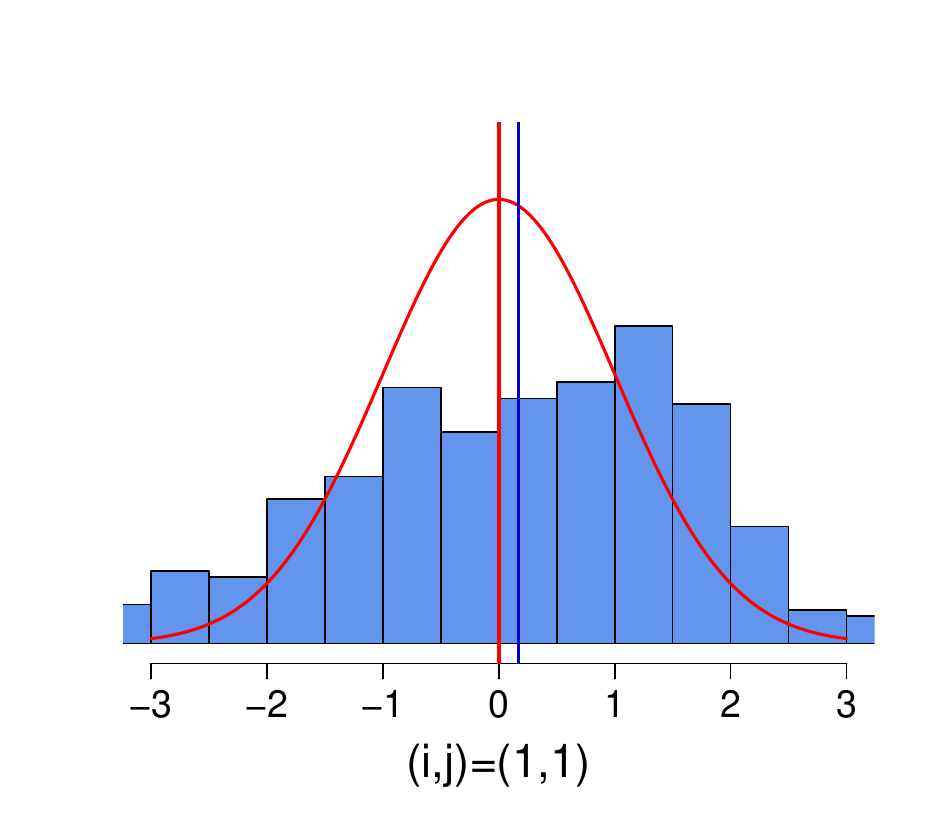}
    \end{minipage}
    \begin{minipage}{0.24\linewidth}
        \centering
        \includegraphics[width=\textwidth]{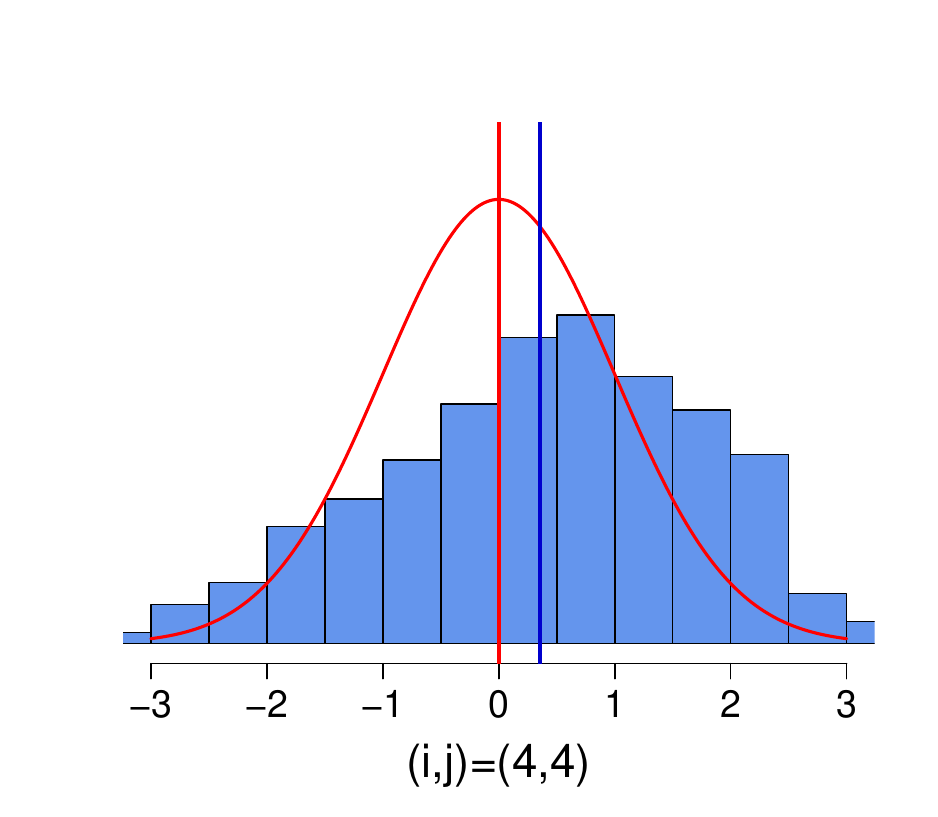}
    \end{minipage}
    \begin{minipage}{0.24\linewidth}
        \centering
        \includegraphics[width=\textwidth]{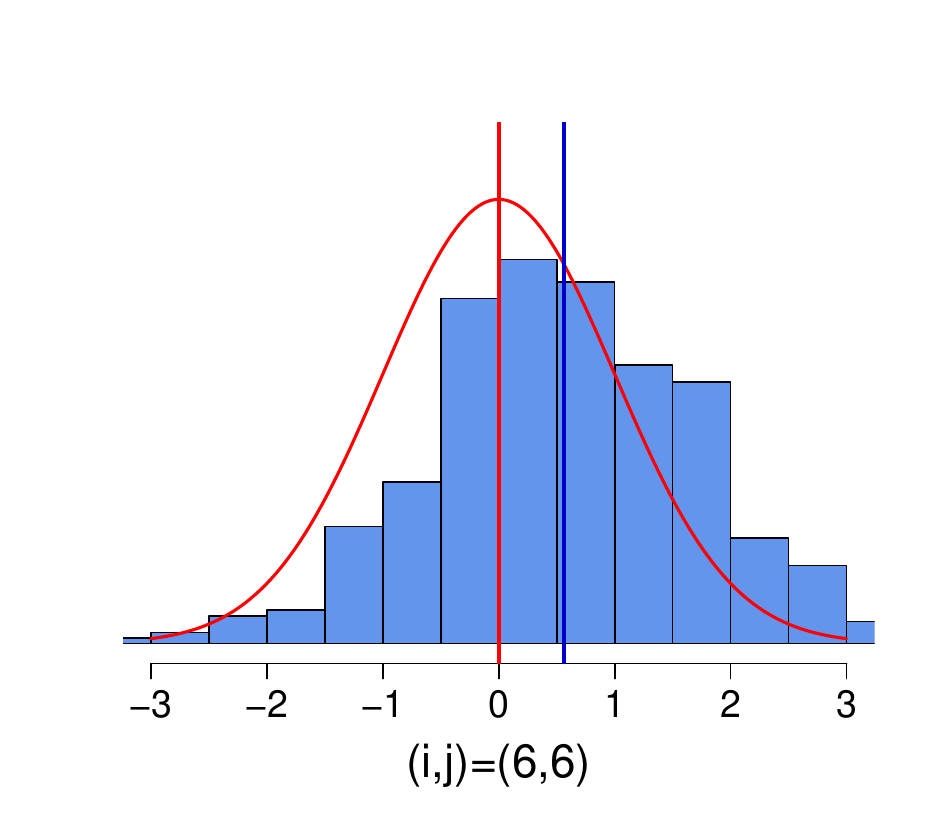}
    \end{minipage}
    \begin{minipage}{0.24\linewidth}
        \centering
        \includegraphics[width=\textwidth]{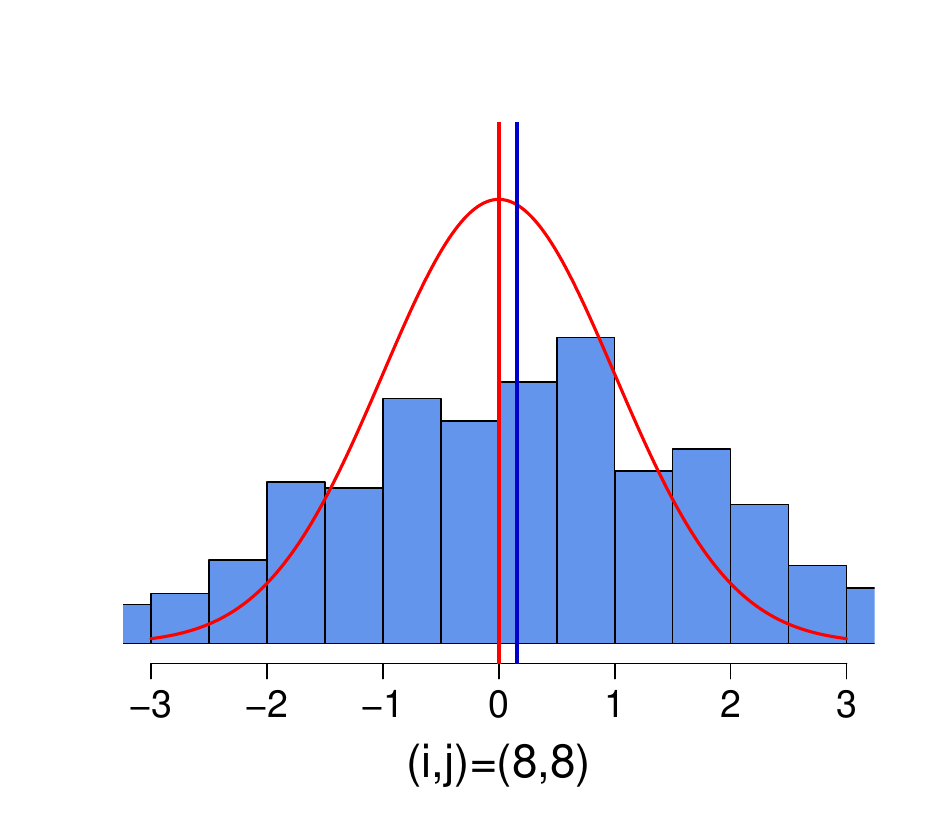}
    \end{minipage}
 \end{minipage}
 \hspace{1cm}
 \begin{minipage}{0.3\linewidth}
    \begin{minipage}{0.24\linewidth}
        \centering
        \includegraphics[width=\textwidth]{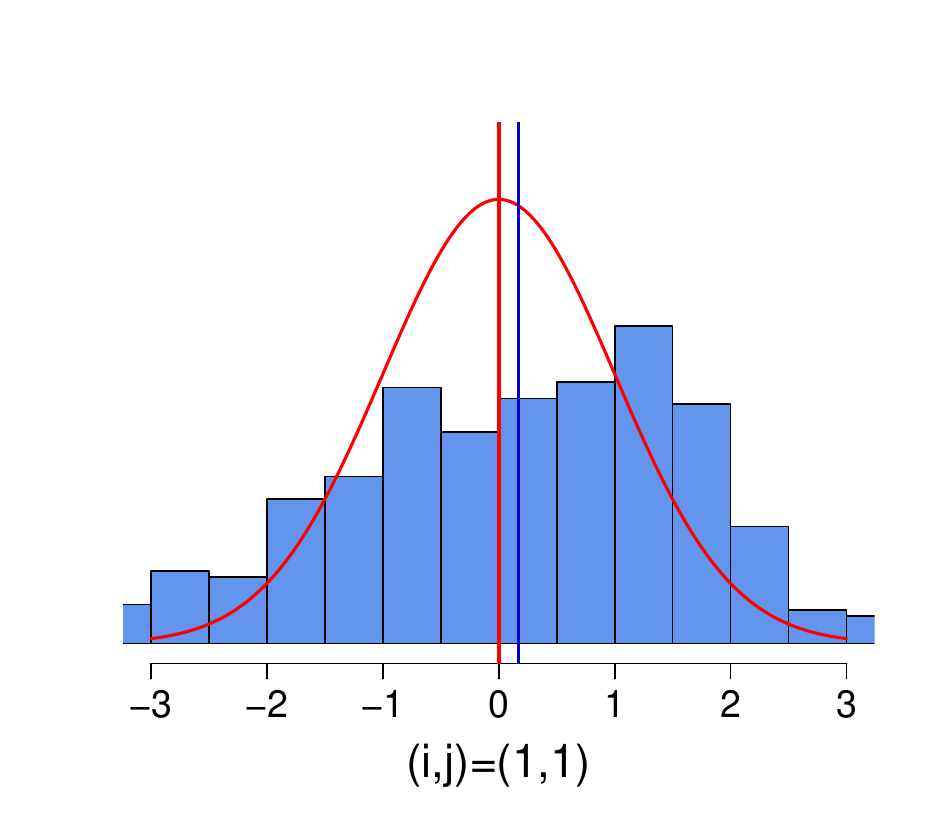}
    \end{minipage}
    \begin{minipage}{0.24\linewidth}
        \centering
        \includegraphics[width=\textwidth]{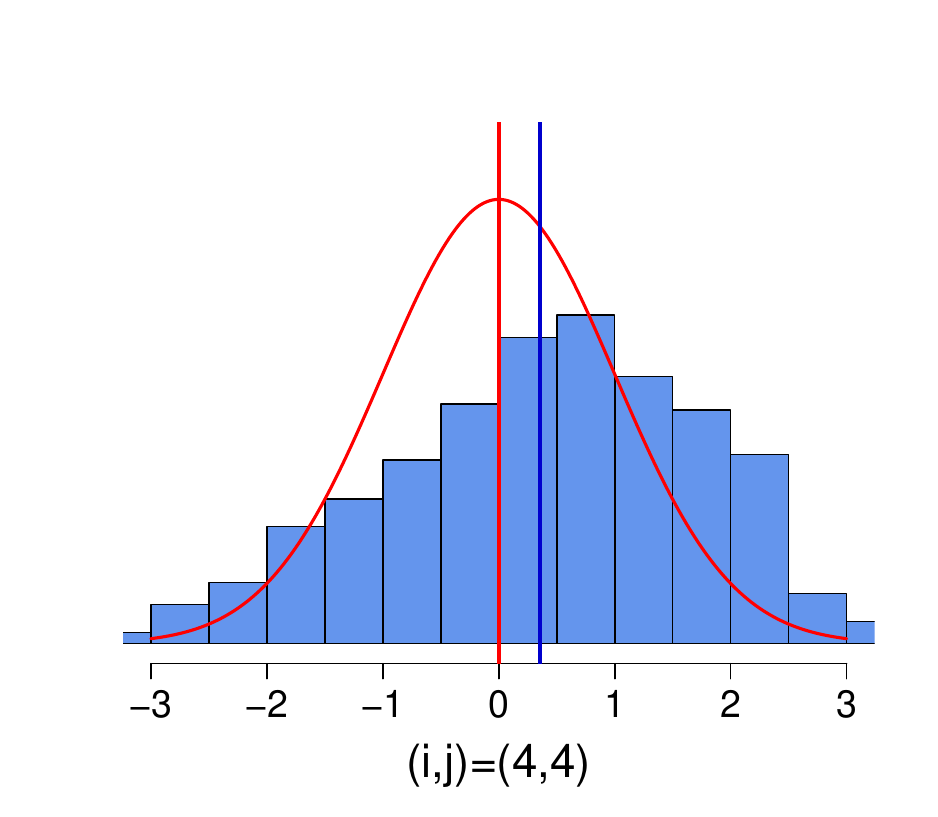}
    \end{minipage}
    \begin{minipage}{0.24\linewidth}
        \centering
        \includegraphics[width=\textwidth]{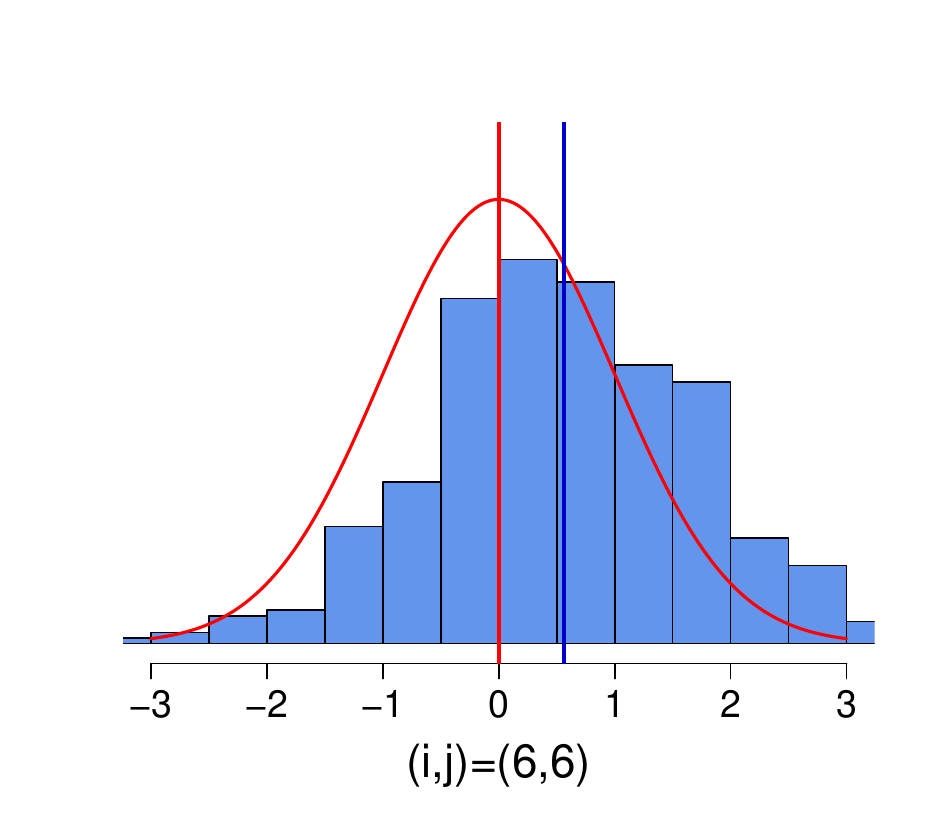}
    \end{minipage}
    \begin{minipage}{0.24\linewidth}
        \centering
        \includegraphics[width=\textwidth]{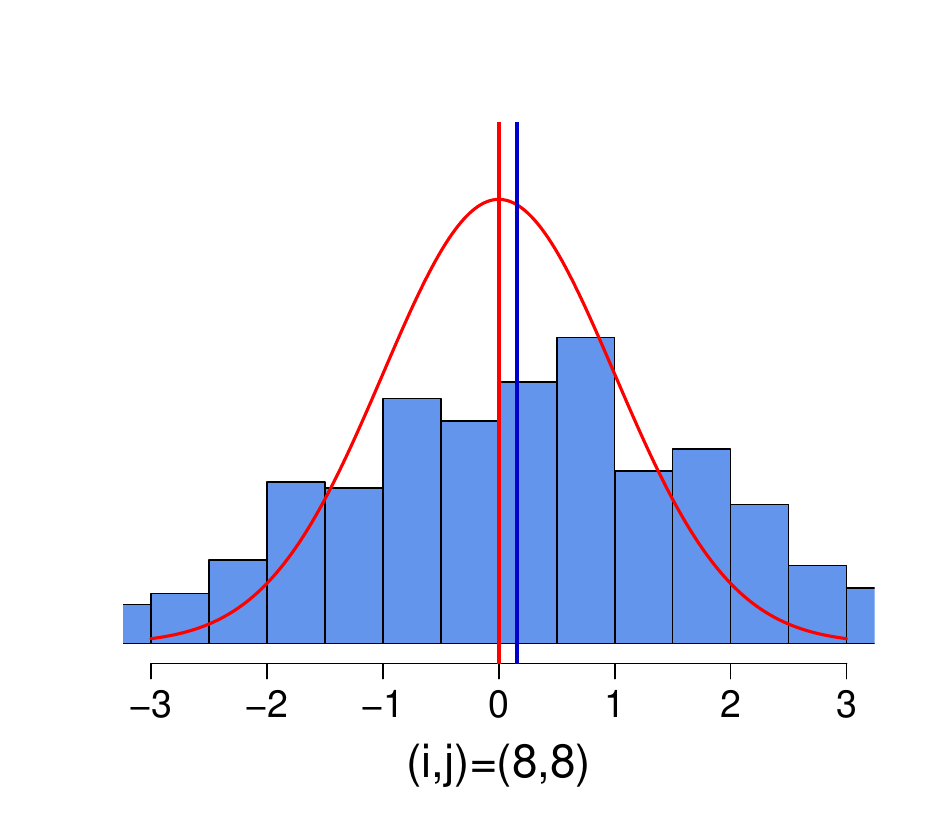}
    \end{minipage}    
 \end{minipage}
  \hspace{1cm}
 \begin{minipage}{0.3\linewidth}
     \begin{minipage}{0.24\linewidth}
        \centering
        \includegraphics[width=\textwidth]{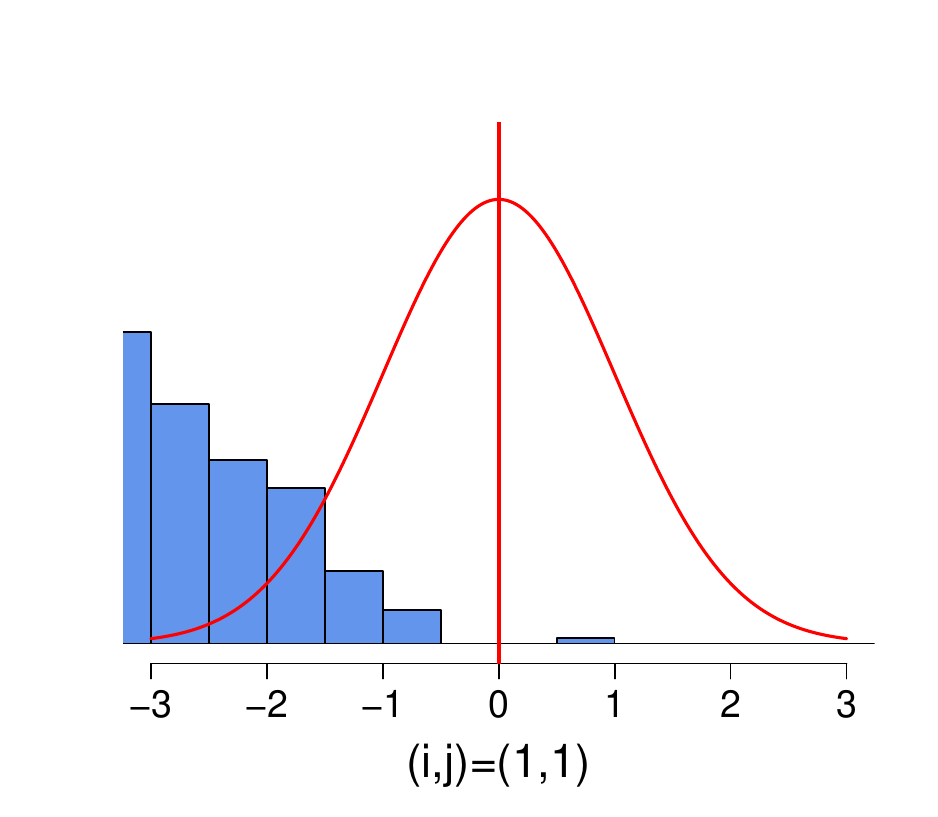}
    \end{minipage}
    \begin{minipage}{0.24\linewidth}
        \centering
        \includegraphics[width=\textwidth]{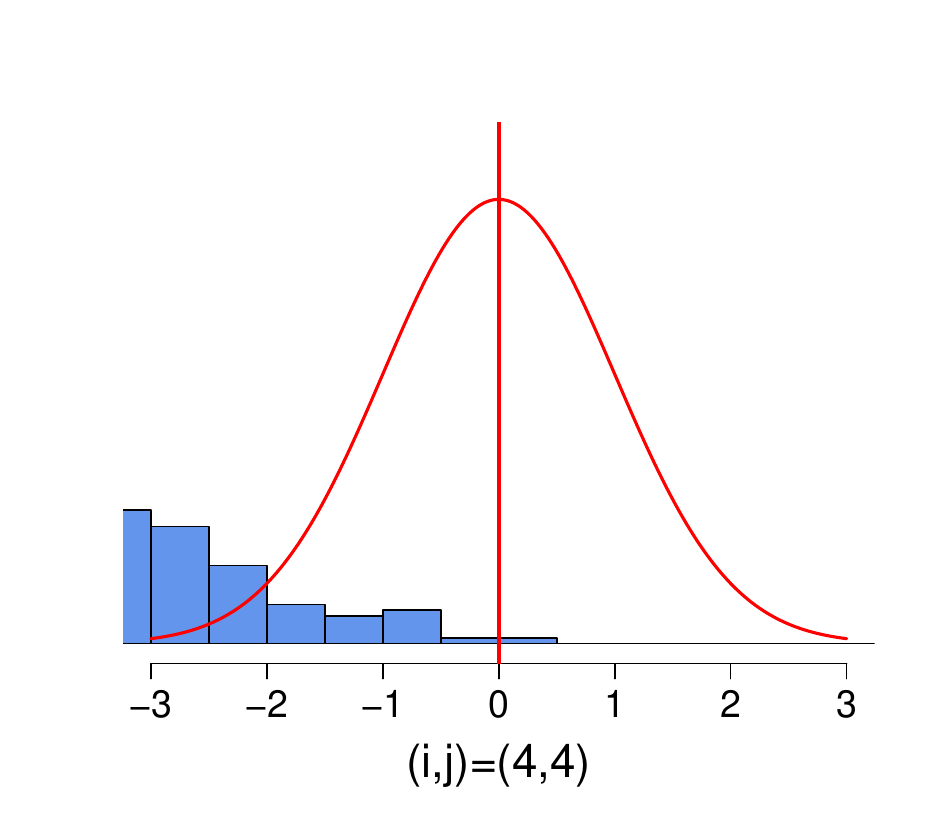}
    \end{minipage}
    \begin{minipage}{0.24\linewidth}
        \centering
        \includegraphics[width=\textwidth]{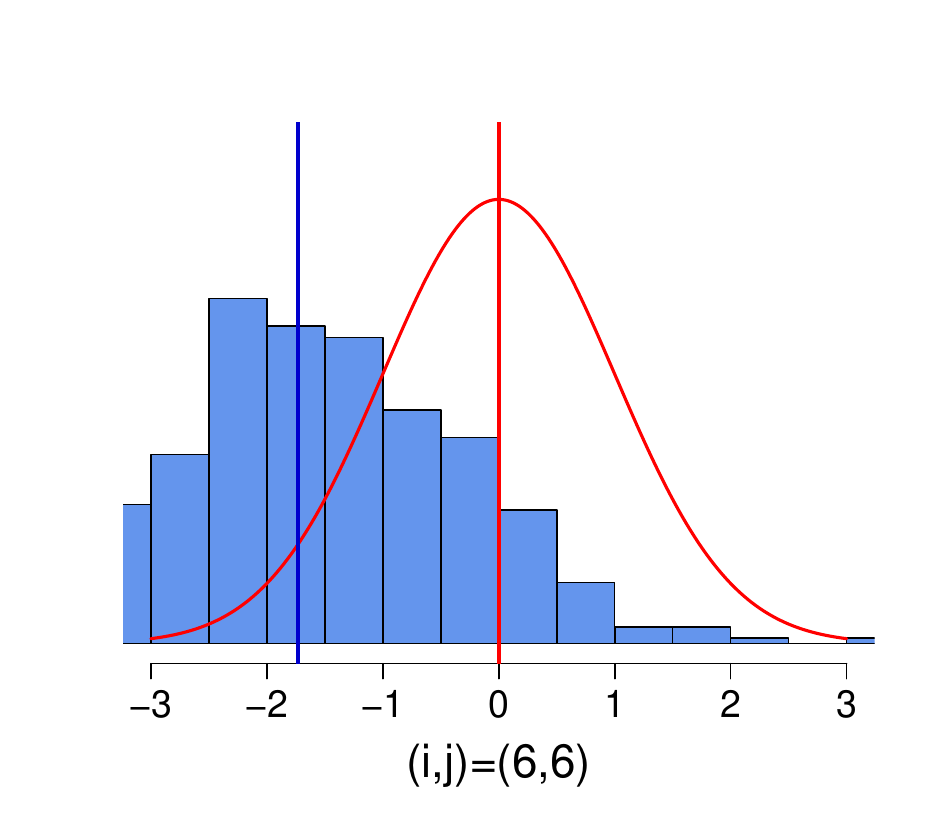}
    \end{minipage}
    \begin{minipage}{0.24\linewidth}
        \centering
        \includegraphics[width=\textwidth]{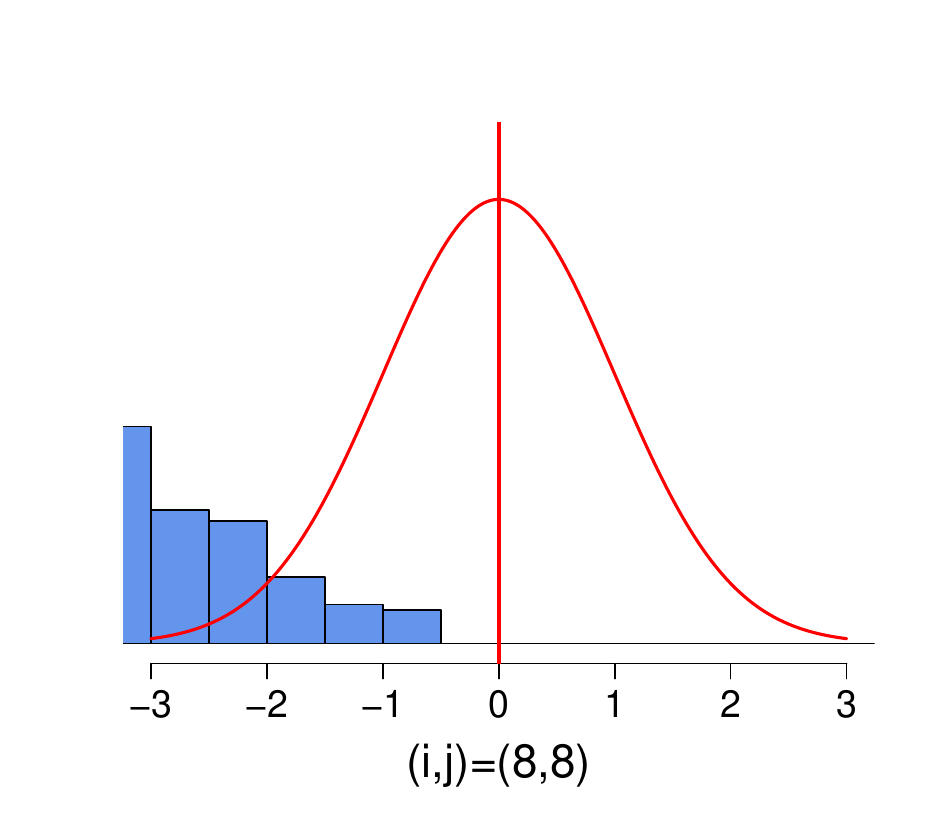}
    \end{minipage}
 \end{minipage}

  \caption*{$n=400, p=400$}
      \vspace{-0.43cm}
 \begin{minipage}{0.3\linewidth}
    \begin{minipage}{0.24\linewidth}
        \centering
        \includegraphics[width=\textwidth]{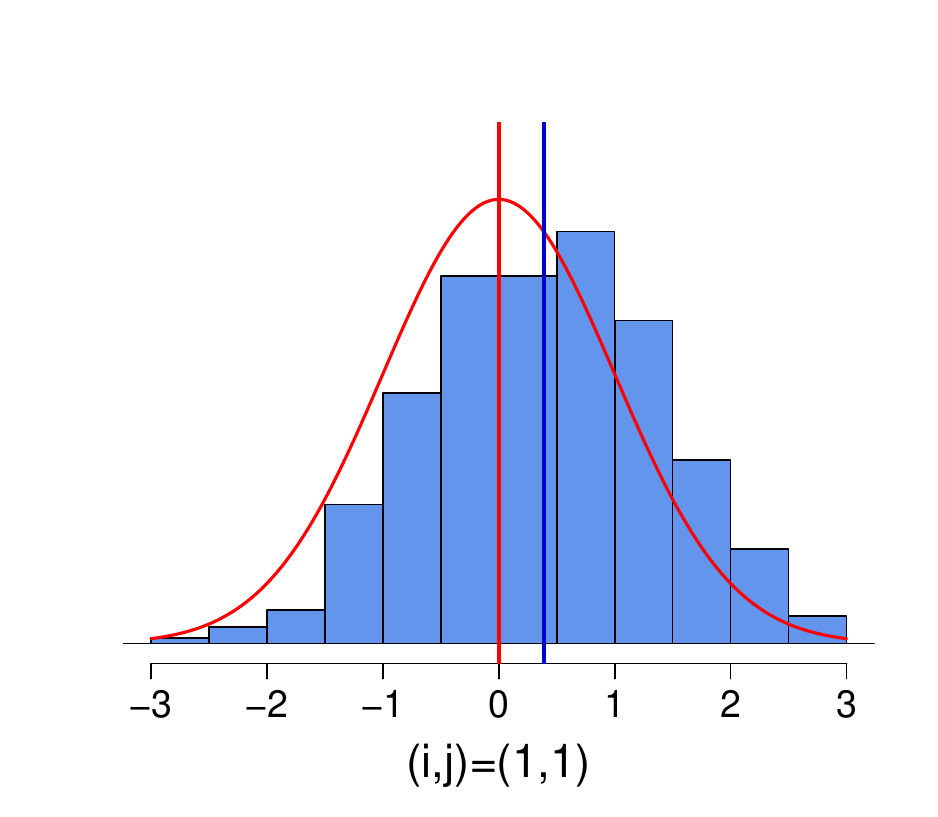}
    \end{minipage}
    \begin{minipage}{0.24\linewidth}
        \centering
        \includegraphics[width=\textwidth]{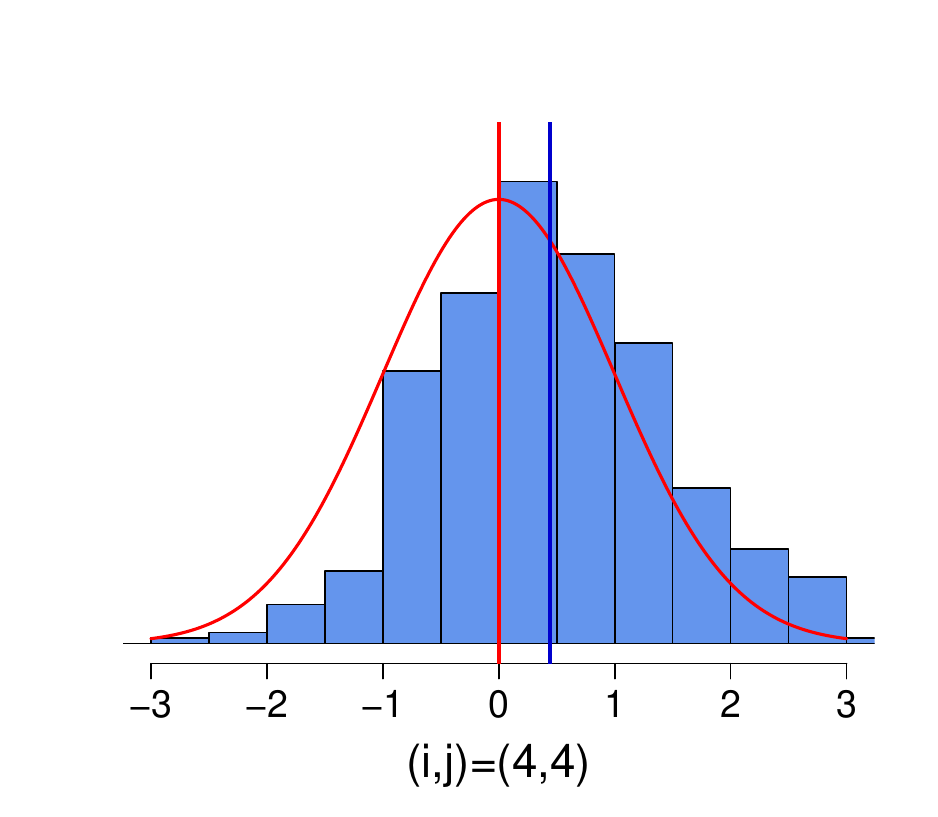}
    \end{minipage}
    \begin{minipage}{0.24\linewidth}
        \centering
        \includegraphics[width=\textwidth]{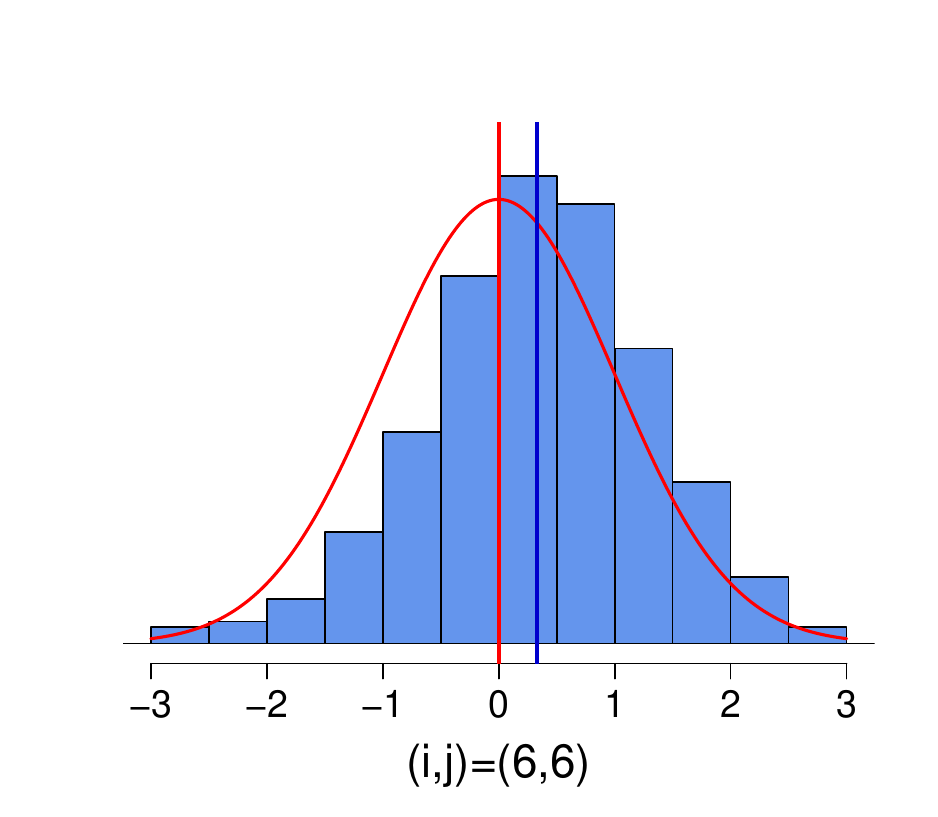}
    \end{minipage}
    \begin{minipage}{0.24\linewidth}
        \centering
        \includegraphics[width=\textwidth]{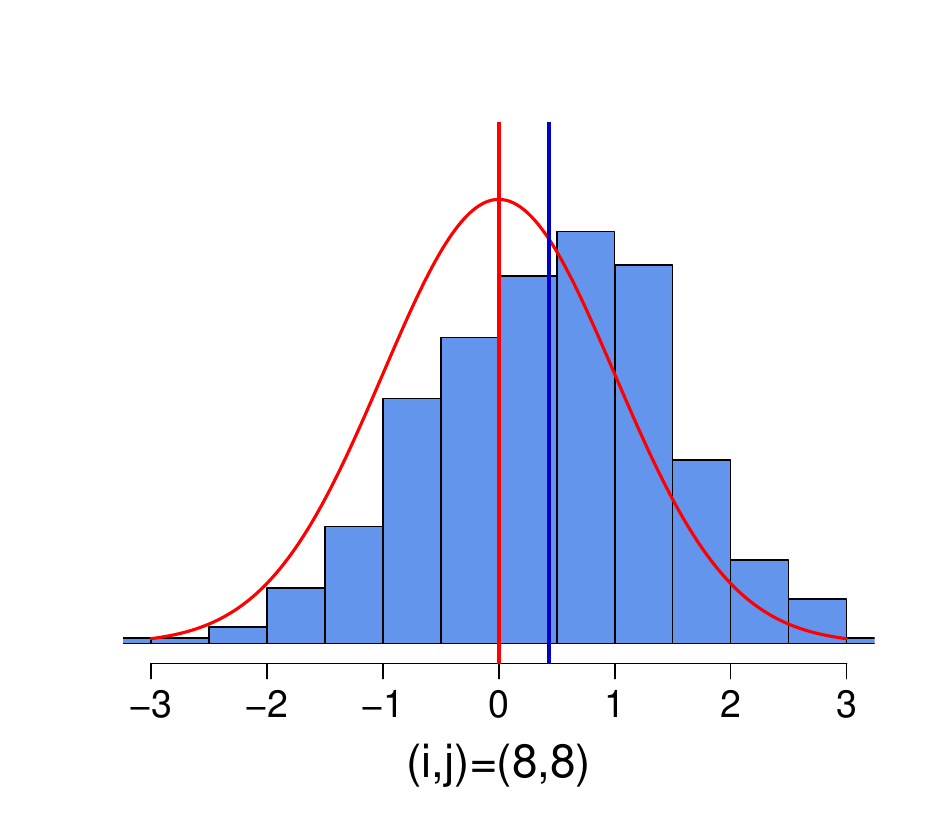}
    \end{minipage}
 \end{minipage}  
     \hspace{1cm}
 \begin{minipage}{0.3\linewidth}
    \begin{minipage}{0.24\linewidth}
        \centering
        \includegraphics[width=\textwidth]{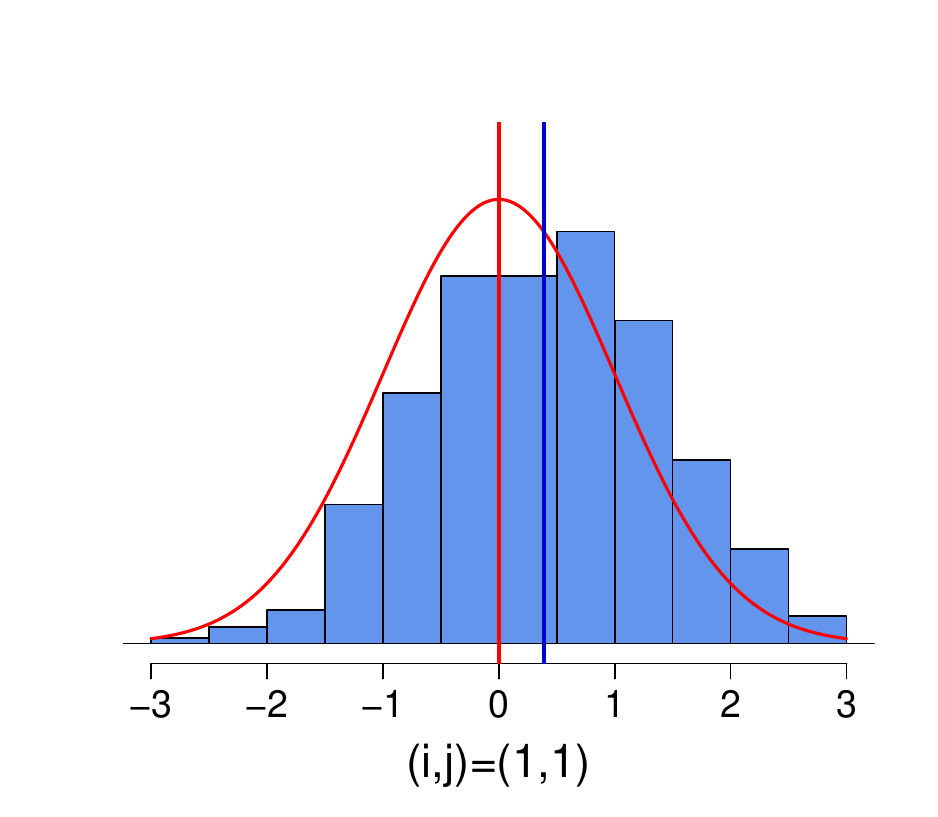}
    \end{minipage}
    \begin{minipage}{0.24\linewidth}
        \centering
        \includegraphics[width=\textwidth]{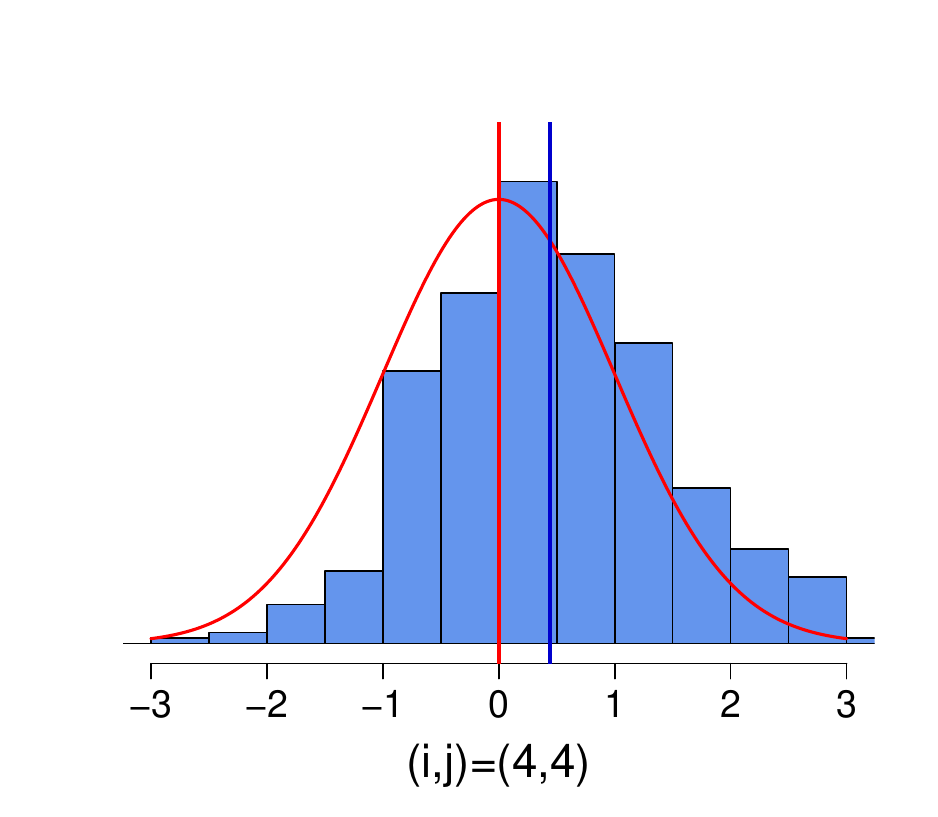}
    \end{minipage}
    \begin{minipage}{0.24\linewidth}
        \centering
        \includegraphics[width=\textwidth]{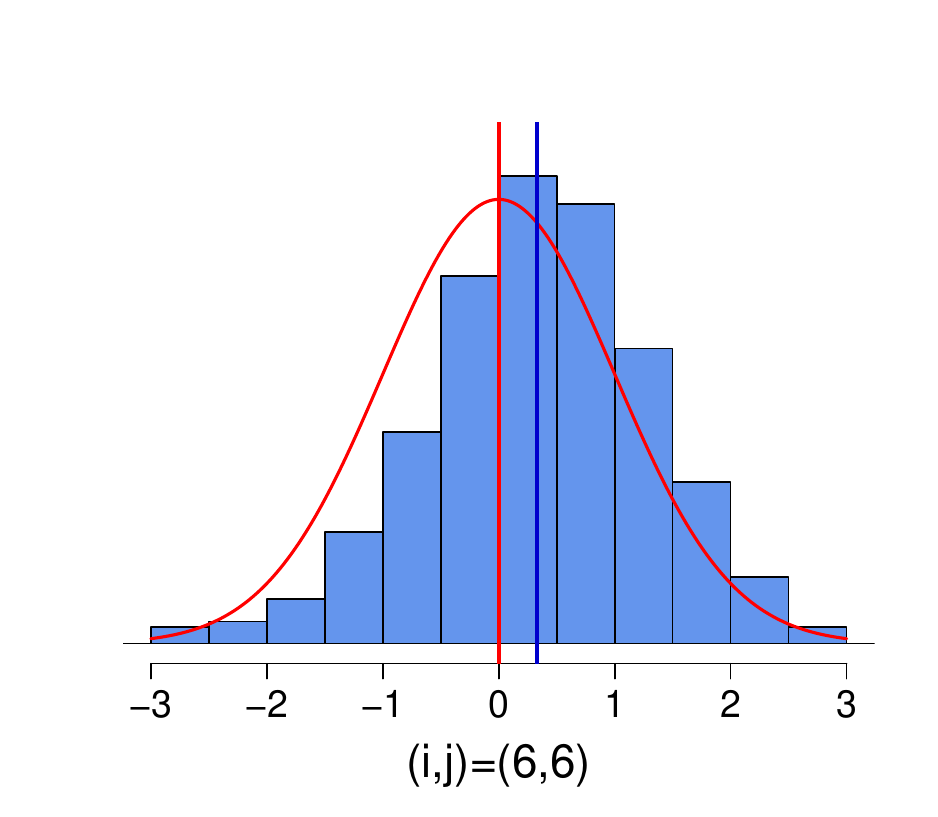}
    \end{minipage}
    \begin{minipage}{0.24\linewidth}
        \centering
        \includegraphics[width=\textwidth]{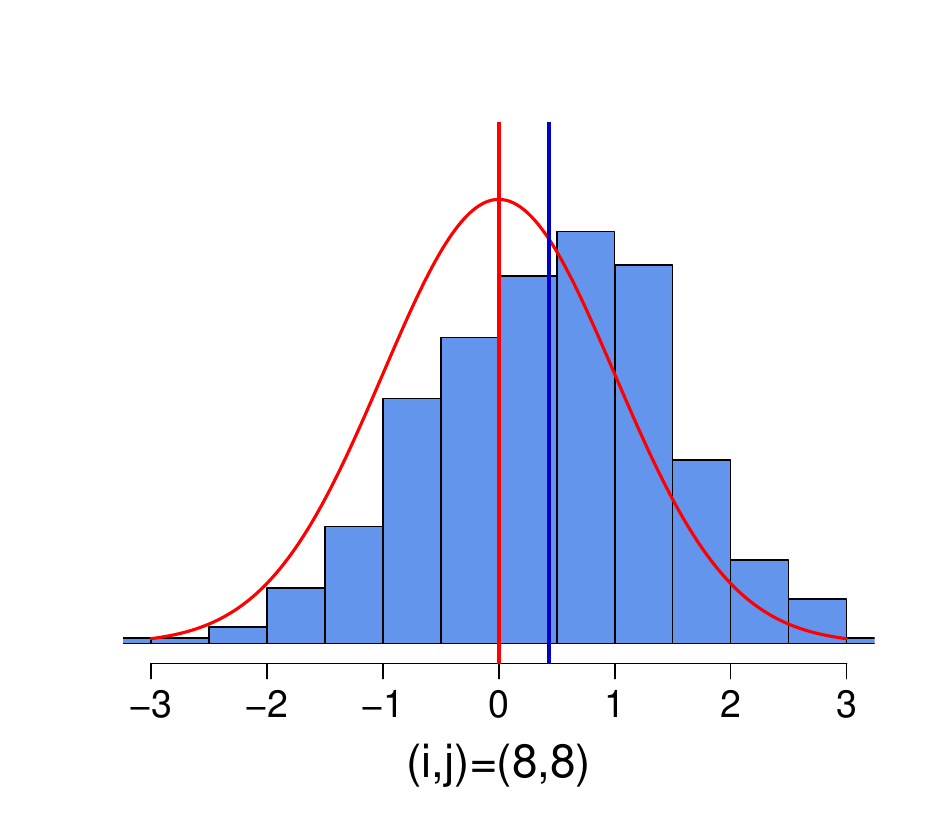}
    \end{minipage}
  \end{minipage}  
    \hspace{1cm}
 \begin{minipage}{0.3\linewidth}
    \begin{minipage}{0.24\linewidth}
        \centering
        \includegraphics[width=\textwidth]{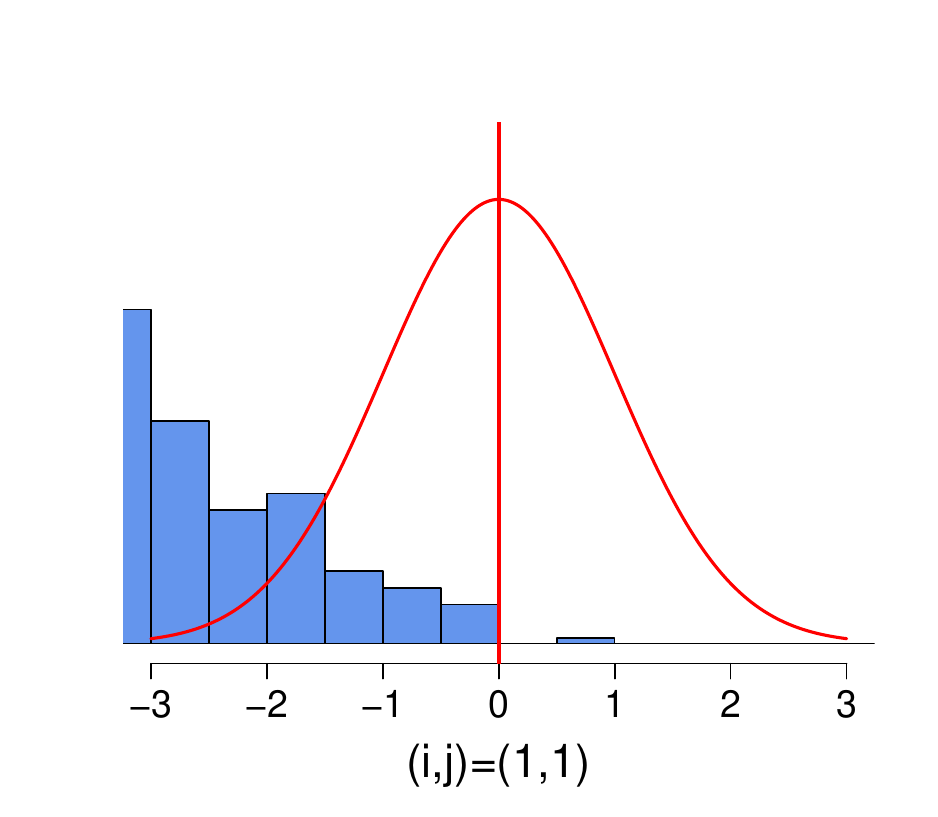}
    \end{minipage}
    \begin{minipage}{0.24\linewidth}
        \centering
        \includegraphics[width=\textwidth]{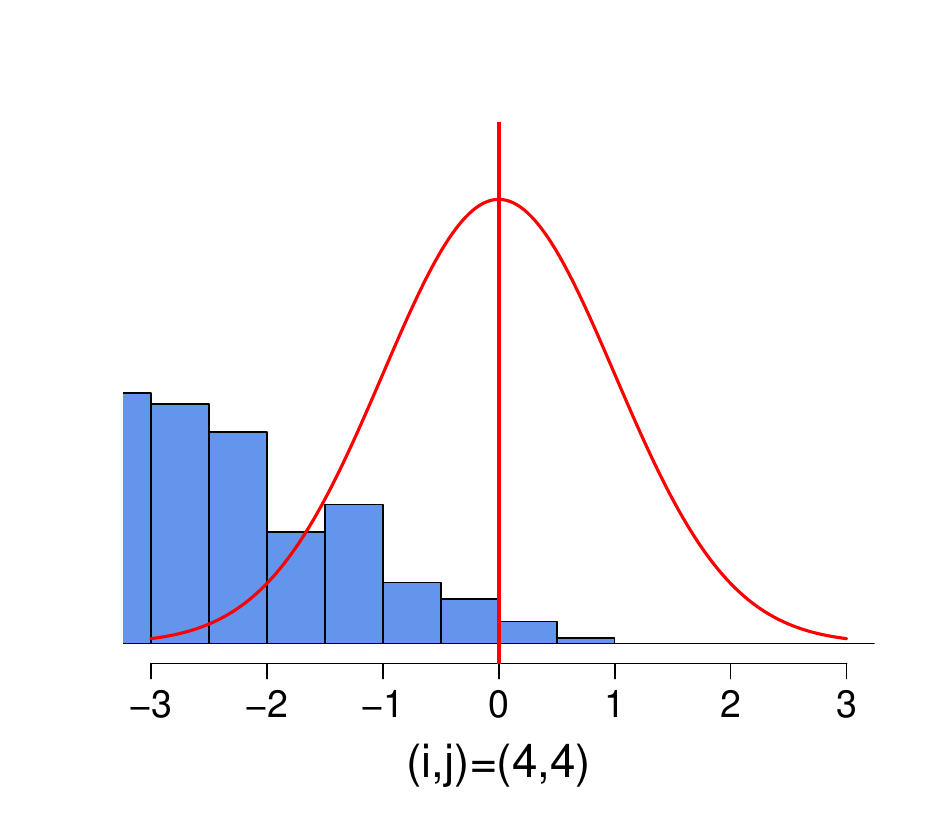}
    \end{minipage}
    \begin{minipage}{0.24\linewidth}
        \centering
        \includegraphics[width=\textwidth]{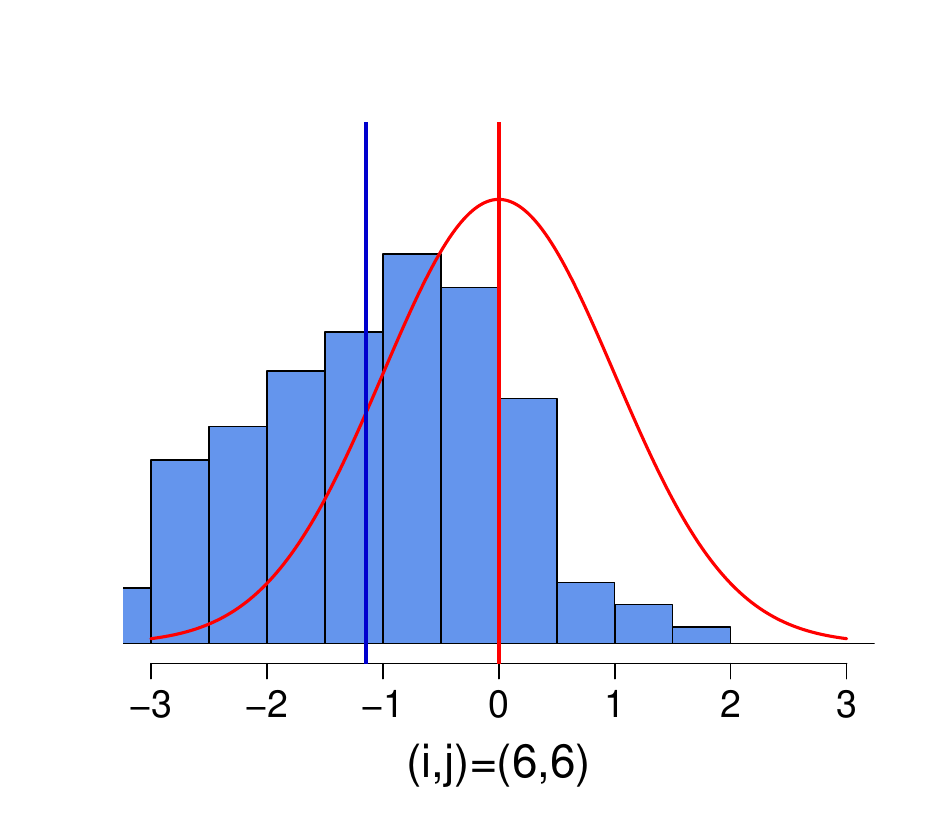}
    \end{minipage}
    \begin{minipage}{0.24\linewidth}
        \centering
        \includegraphics[width=\textwidth]{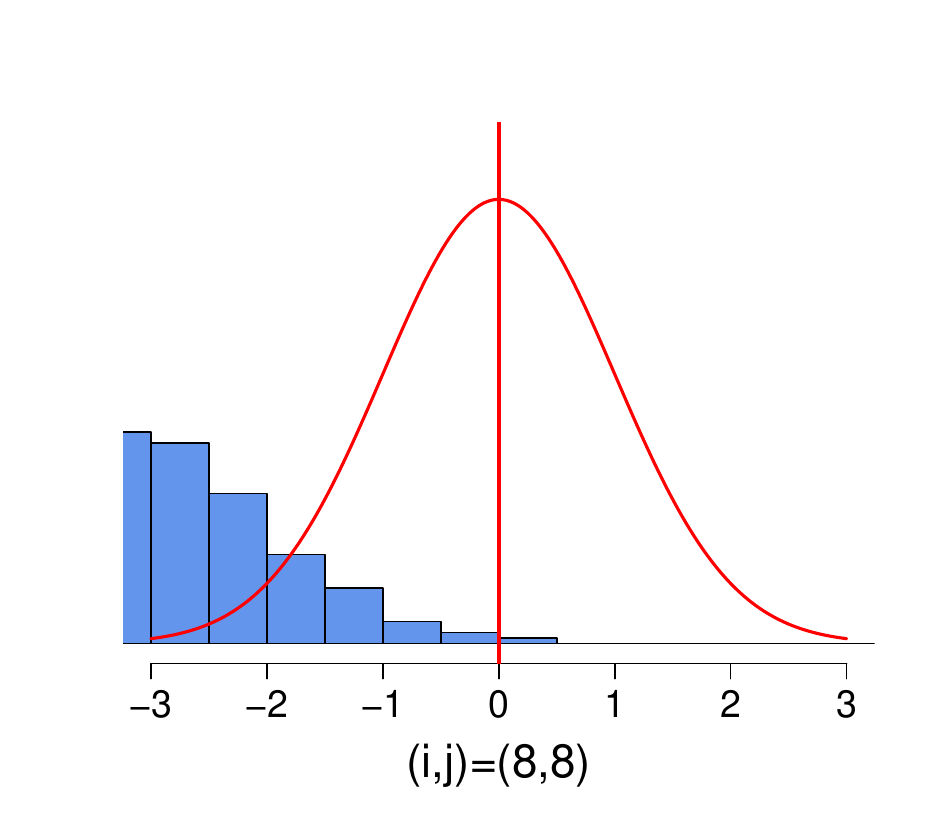}
    \end{minipage}
 \end{minipage}

  \caption*{$n=800, p=400$}
      \vspace{-0.43cm}
 \begin{minipage}{0.3\linewidth}
    \begin{minipage}{0.24\linewidth}
        \centering
        \includegraphics[width=\textwidth]{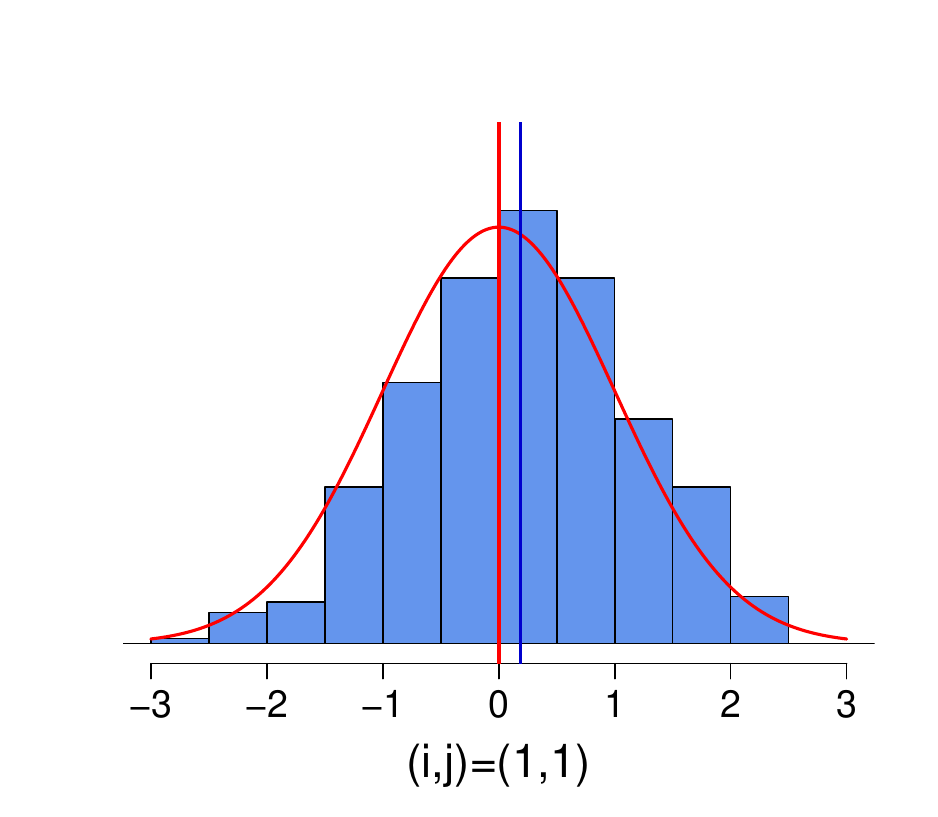}
    \end{minipage}
    \begin{minipage}{0.24\linewidth}
        \centering
        \includegraphics[width=\textwidth]{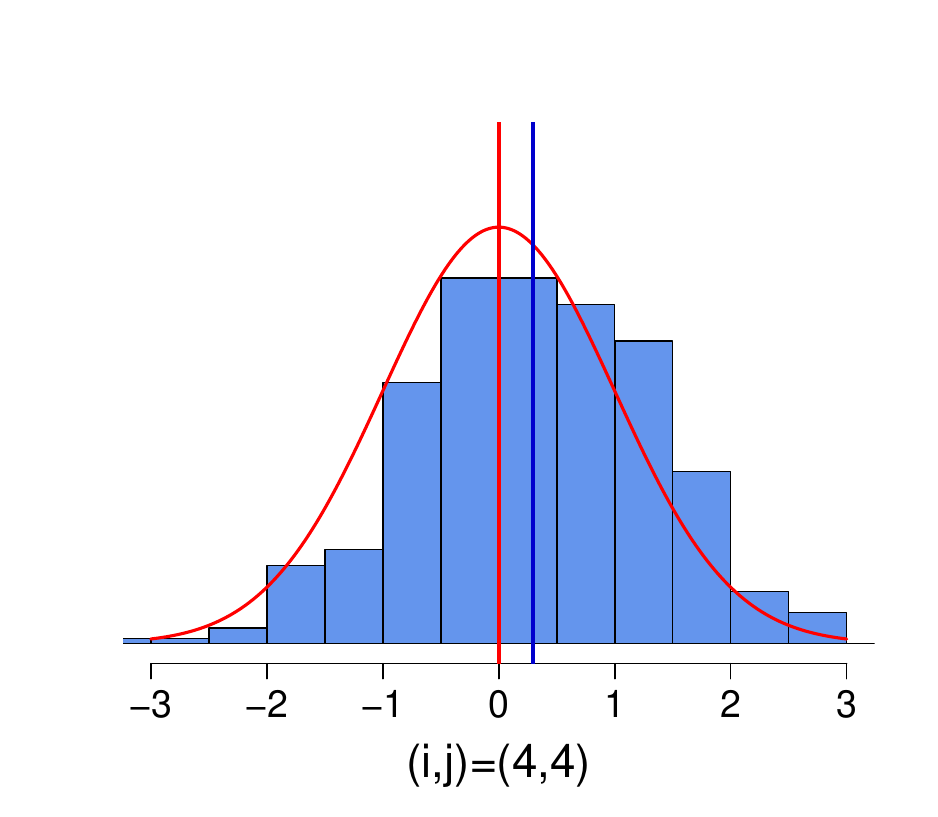}
    \end{minipage}
    \begin{minipage}{0.24\linewidth}
        \centering
        \includegraphics[width=\textwidth]{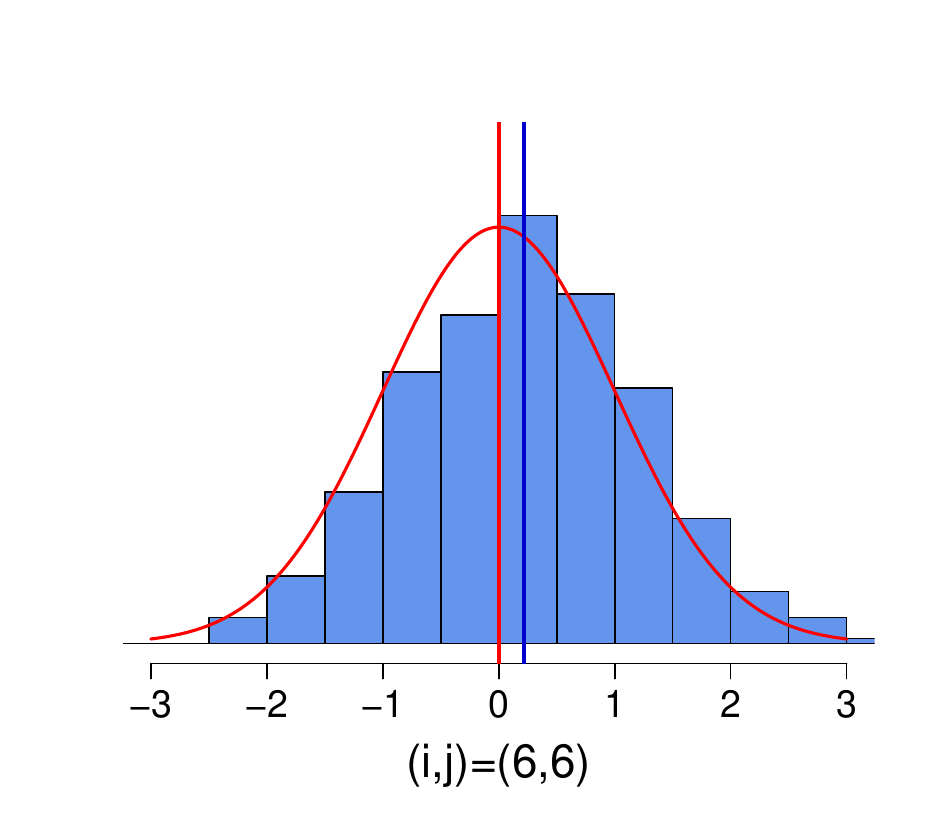}
    \end{minipage}
    \begin{minipage}{0.24\linewidth}
        \centering
        \includegraphics[width=\textwidth]{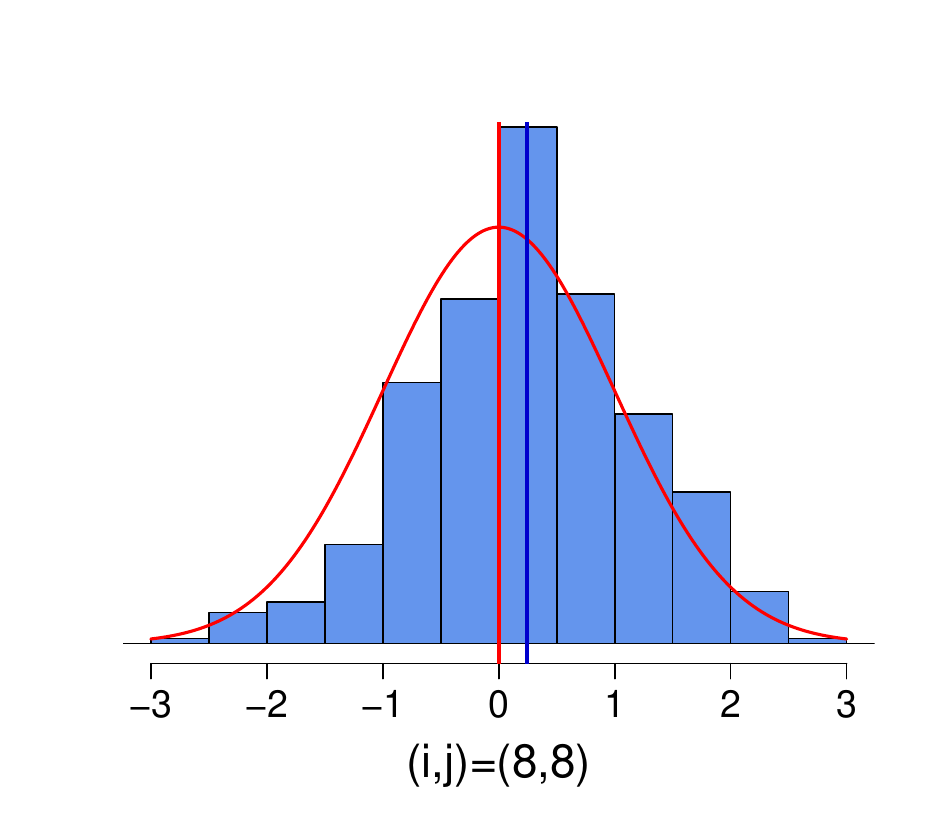}
    \end{minipage}
   \caption*{(a)~~$L_0{:}~ \widehat{\mb{\Omega}}^{\text{US}}$}
 \end{minipage} 
     \hspace{1cm}
 \begin{minipage}{0.3\linewidth}
    \begin{minipage}{0.24\linewidth}
        \centering
        \includegraphics[width=\textwidth]{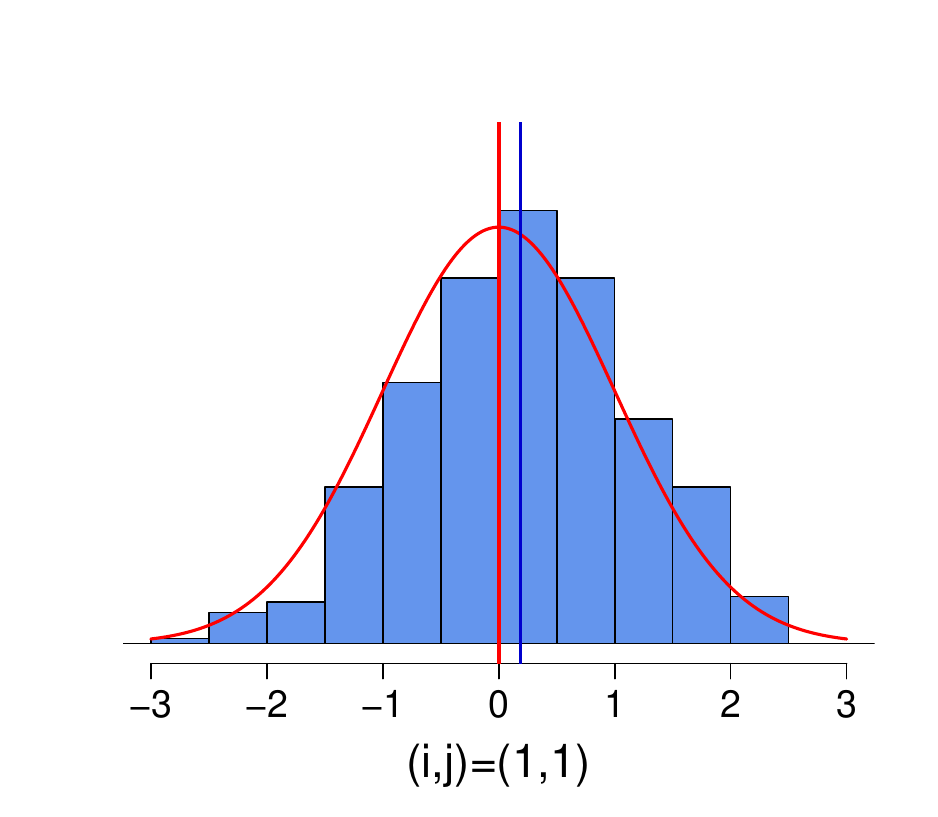}
    \end{minipage}
    \begin{minipage}{0.24\linewidth}
        \centering
        \includegraphics[width=\textwidth]{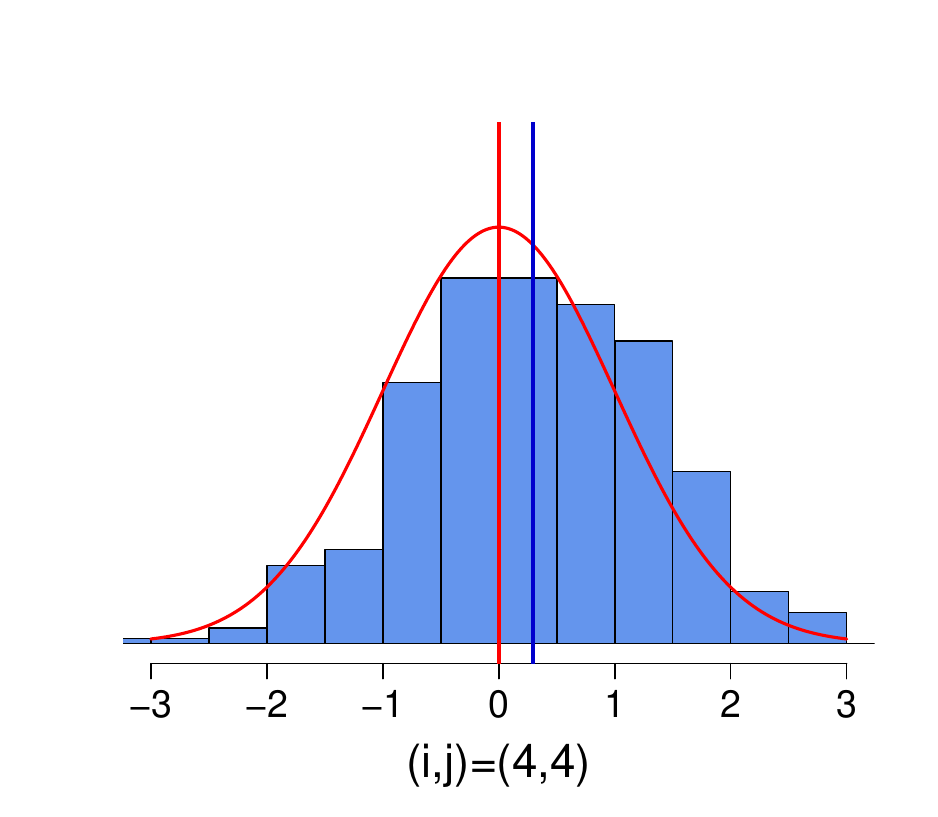}
    \end{minipage}
    \begin{minipage}{0.24\linewidth}
        \centering
        \includegraphics[width=\textwidth]{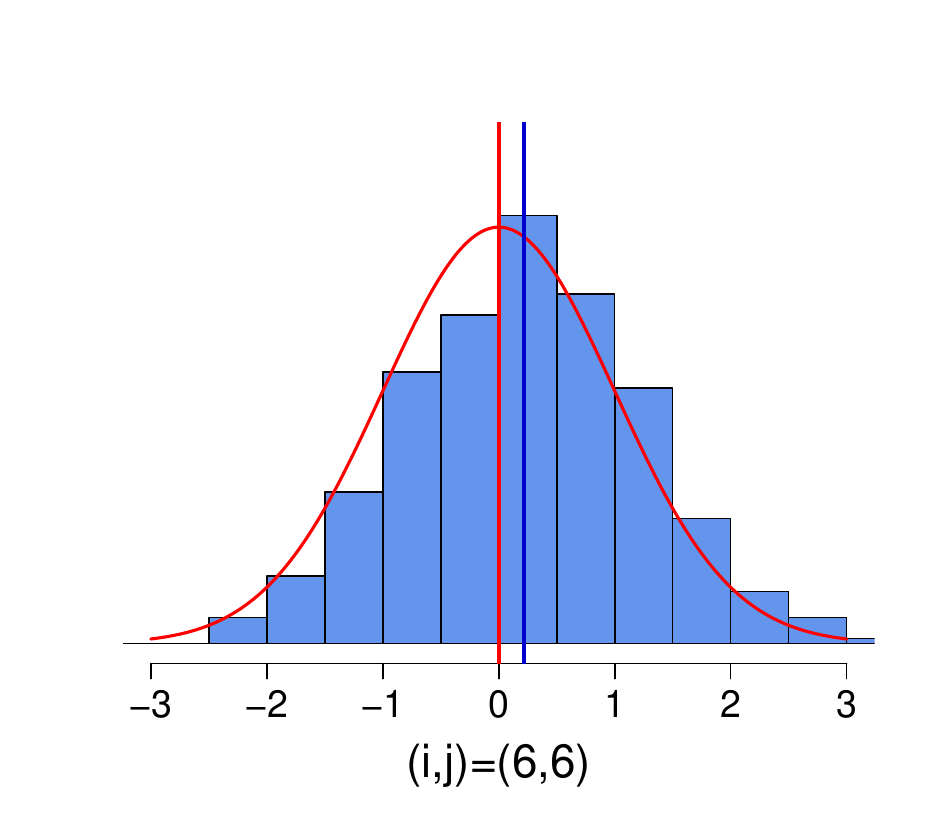}
    \end{minipage}
    \begin{minipage}{0.24\linewidth}
        \centering
        \includegraphics[width=\textwidth]{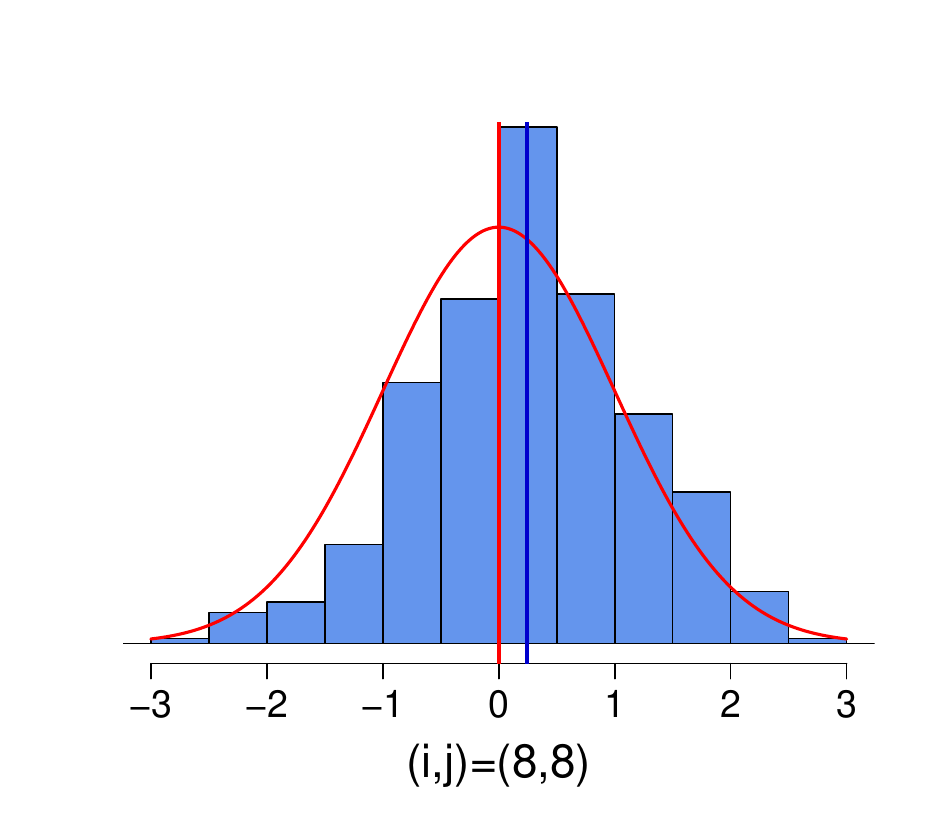}
    \end{minipage}
        \caption*{(b)~~$L_0{:}~ \widehat{\mb{T}}$}
 \end{minipage}   
      \hspace{1cm}
 \begin{minipage}{0.3\linewidth}
    \begin{minipage}{0.24\linewidth}
        \centering
        \includegraphics[width=\textwidth]{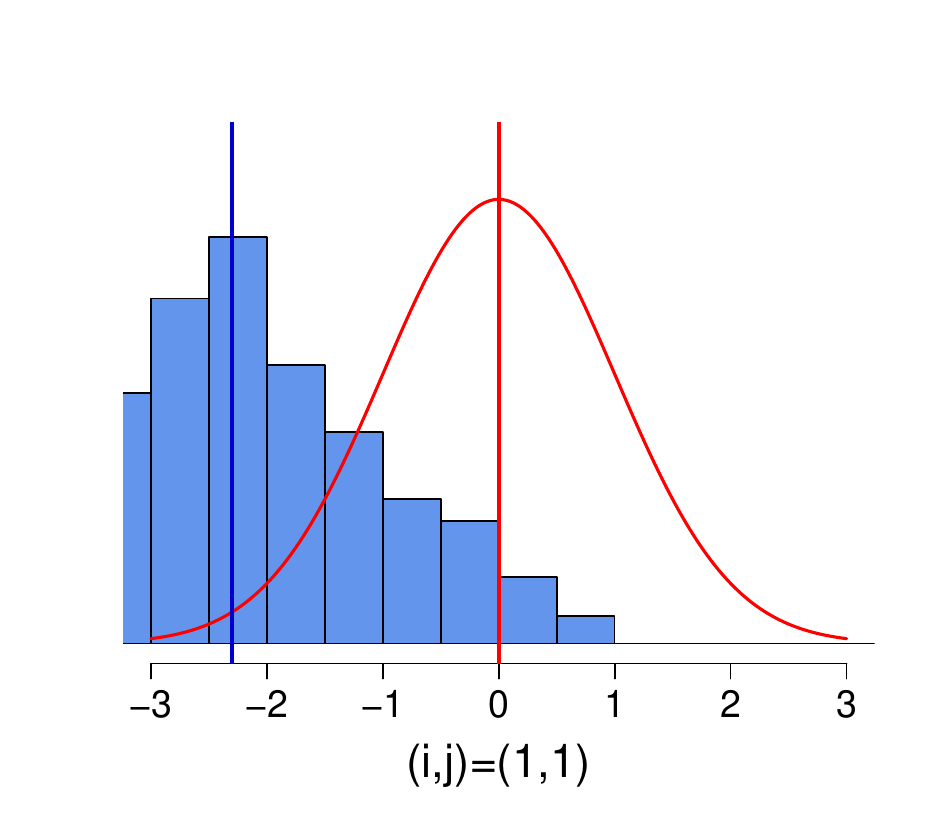}
    \end{minipage}
    \begin{minipage}{0.24\linewidth}
        \centering
        \includegraphics[width=\textwidth]{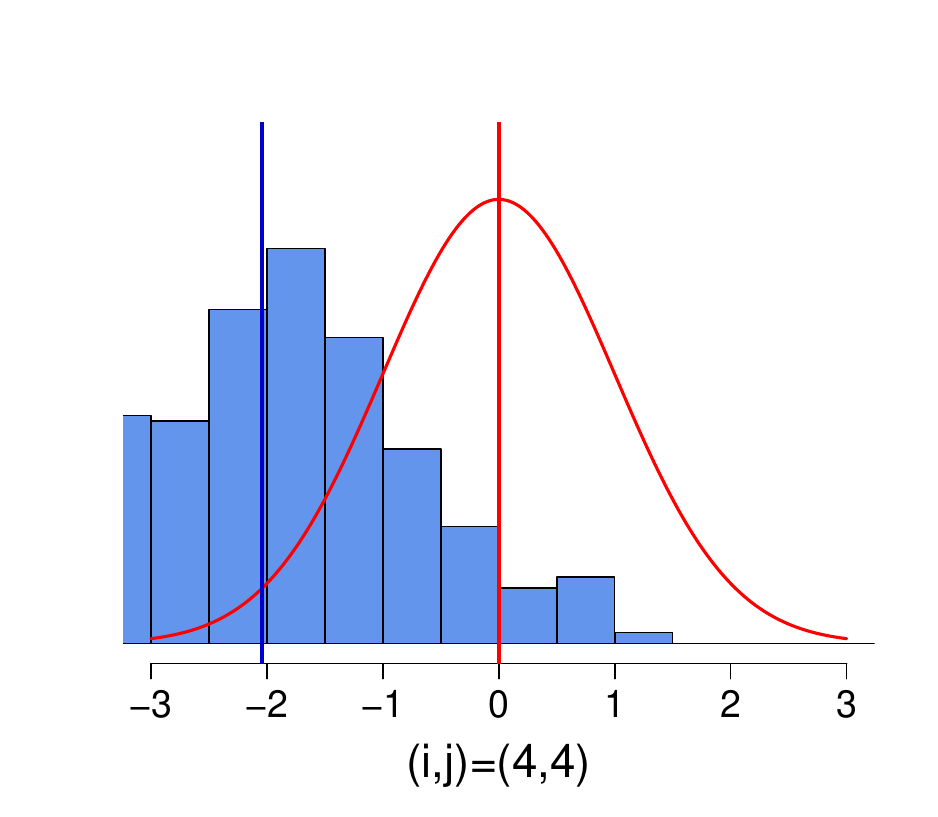}
    \end{minipage}
    \begin{minipage}{0.24\linewidth}
        \centering
        \includegraphics[width=\textwidth]{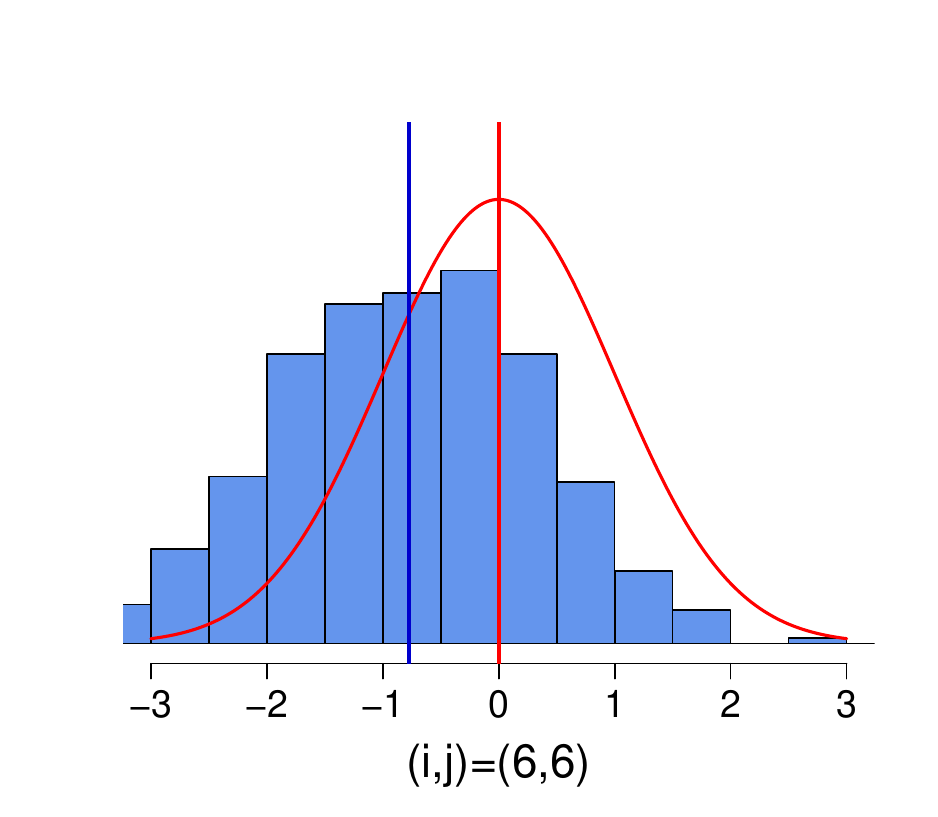}
    \end{minipage}
    \begin{minipage}{0.24\linewidth}
        \centering
        \includegraphics[width=\textwidth]{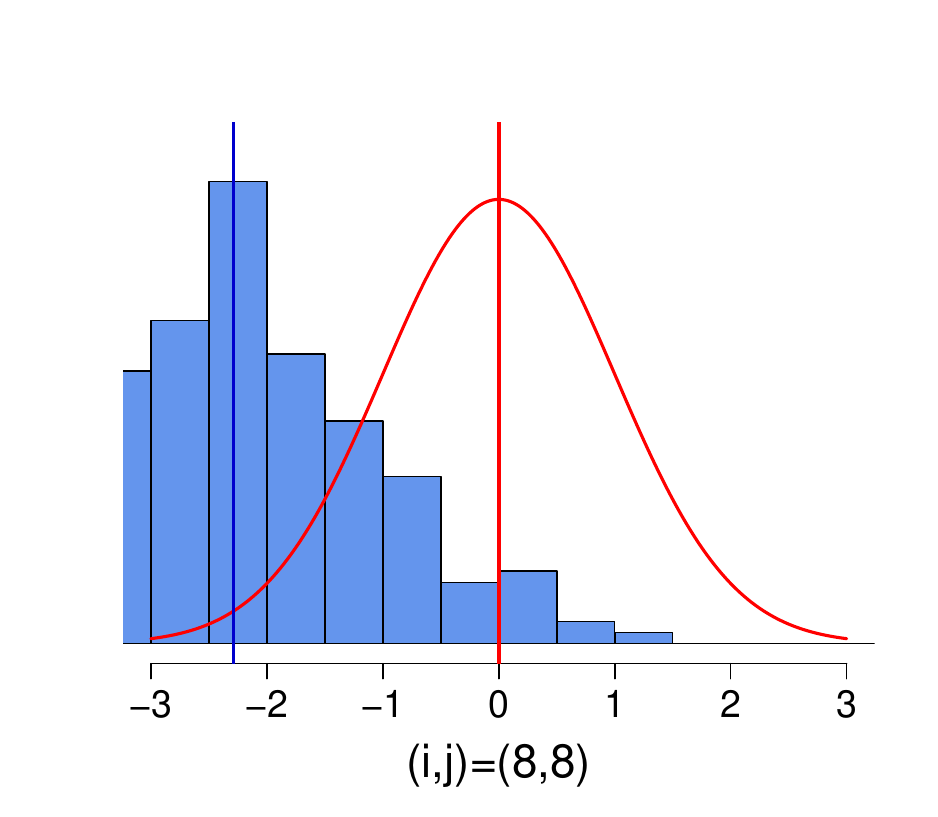}
    \end{minipage}
    \caption*{(c)~~$L_1{:}~ \widehat{\mb{T}}$}
     \end{minipage}   
     \caption{Histograms of $\big(\sqrt{n}(\widehat{\mb{\Omega}}_{ij}^{(m)}-\mb{\Omega}_{ij})/\widehat{\sigma}_{\mb{\Omega}_{ij}}^{(m)}\big)_{m=1}^{400}$ under Gaussian cluster graph settings.}
\end{sidewaysfigure}

%sub-Gaussian band
  \begin{sidewaysfigure}[th!]
  \caption*{$n=200, p=200$}
      \vspace{-0.43cm}
 \begin{minipage}{0.3\linewidth}
    \begin{minipage}{0.24\linewidth}
        \centering
        \includegraphics[width=\textwidth]{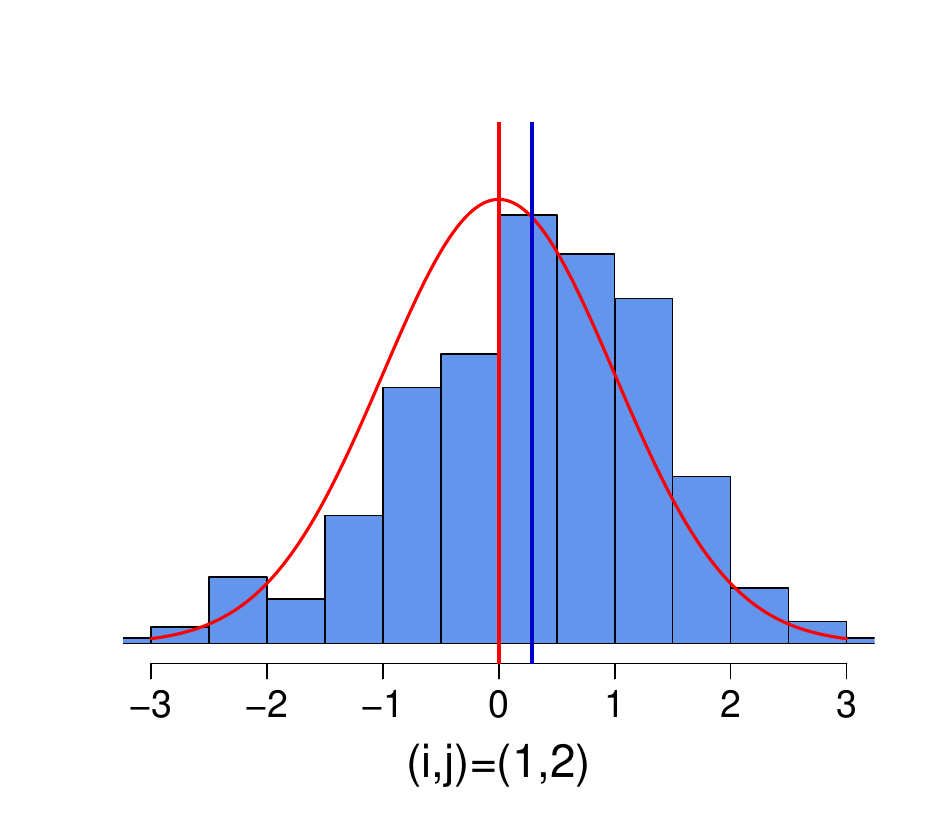}
    \end{minipage}
    \begin{minipage}{0.24\linewidth}
        \centering
        \includegraphics[width=\textwidth]{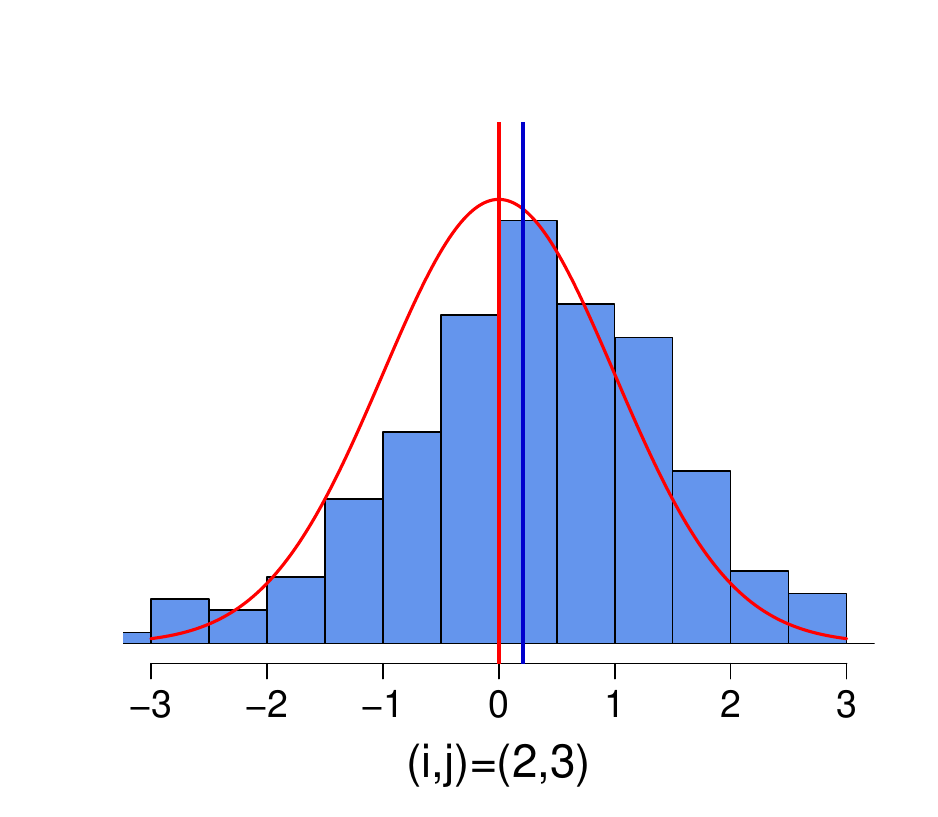}
    \end{minipage}
    \begin{minipage}{0.24\linewidth}
        \centering
        \includegraphics[width=\textwidth]{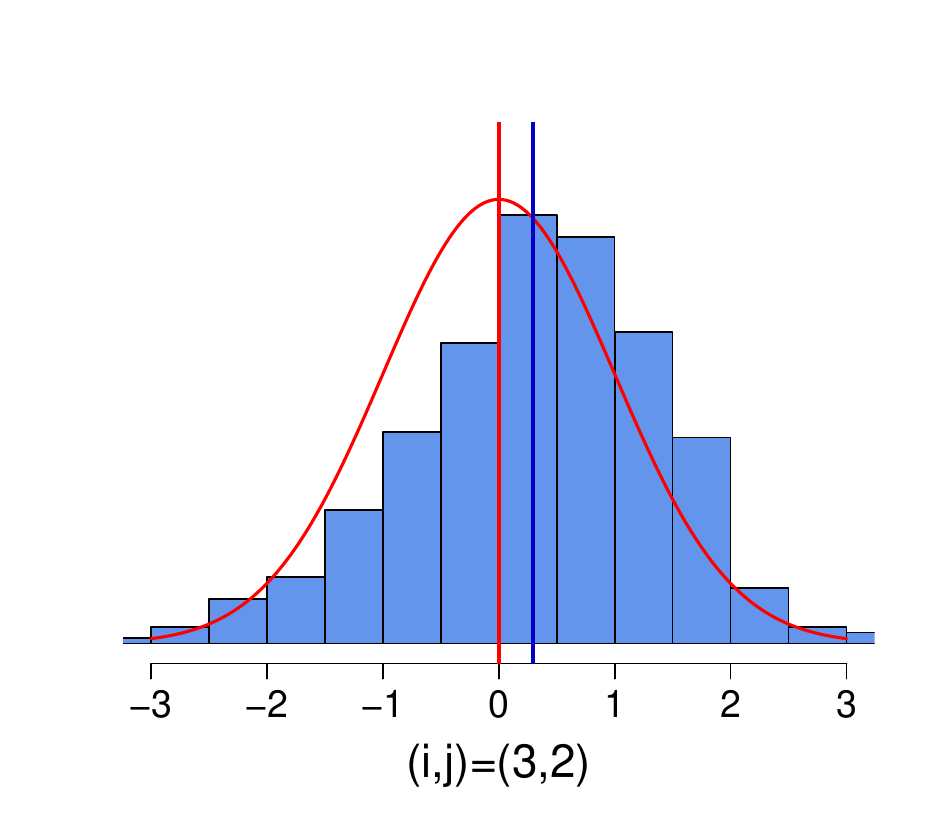}
    \end{minipage}
    \begin{minipage}{0.24\linewidth}
        \centering
        \includegraphics[width=\textwidth]{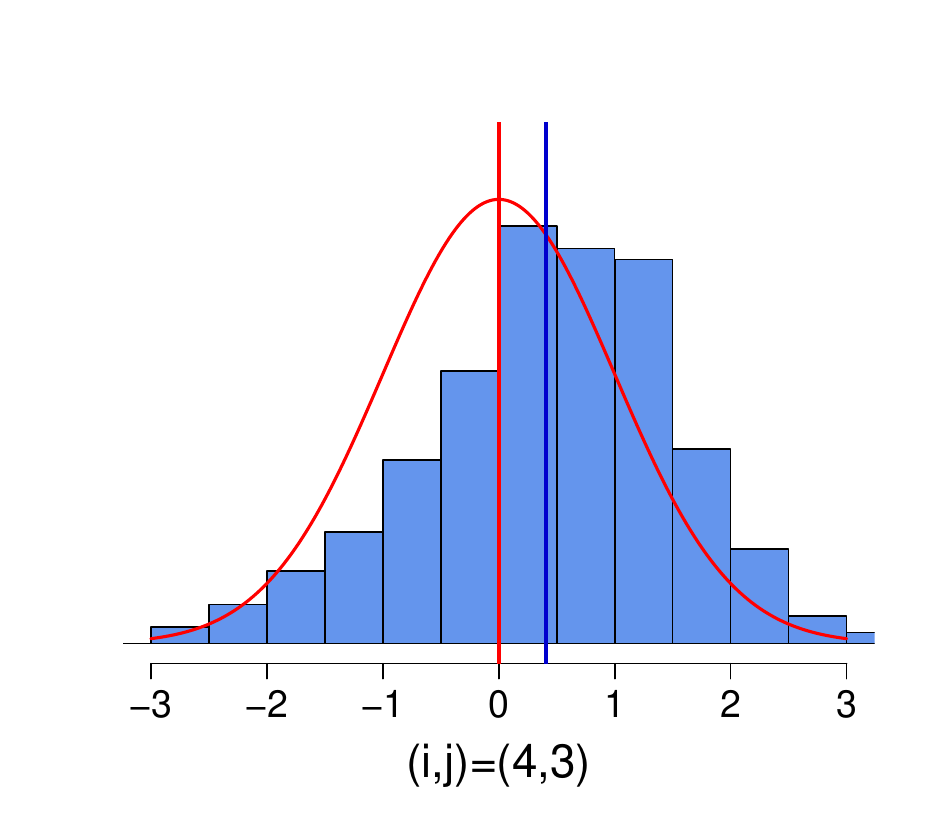}
    \end{minipage}
 \end{minipage}
 \hspace{1cm}
 \begin{minipage}{0.3\linewidth}
    \begin{minipage}{0.24\linewidth}
        \centering
        \includegraphics[width=\textwidth]{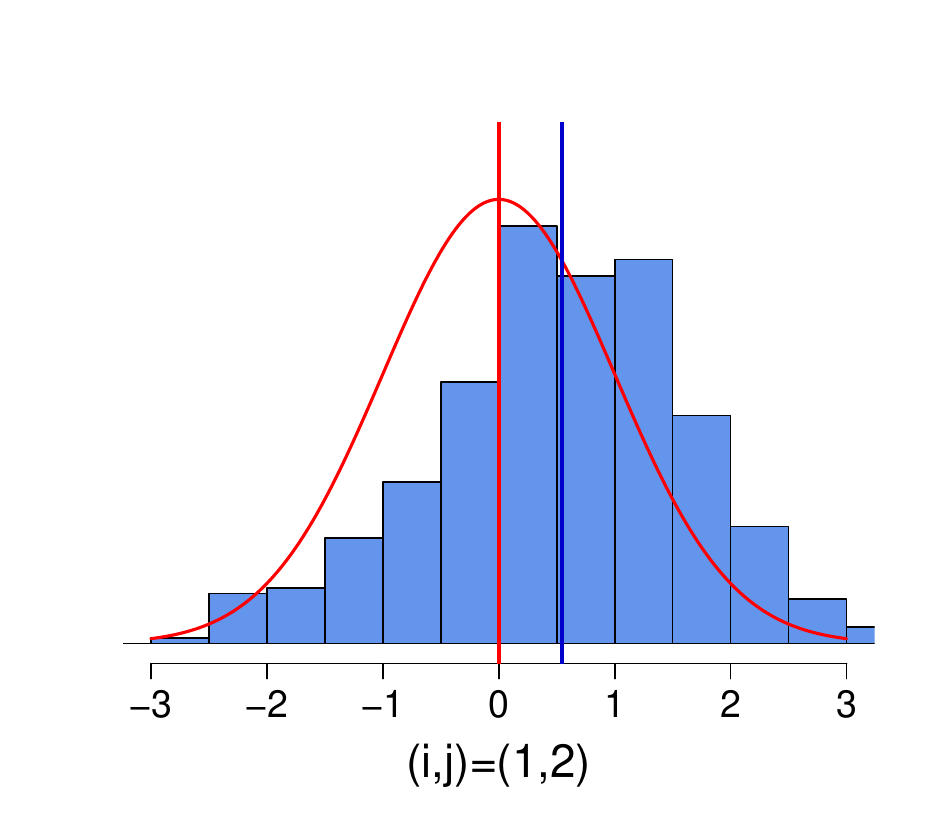}
    \end{minipage}
    \begin{minipage}{0.24\linewidth}
        \centering
        \includegraphics[width=\textwidth]{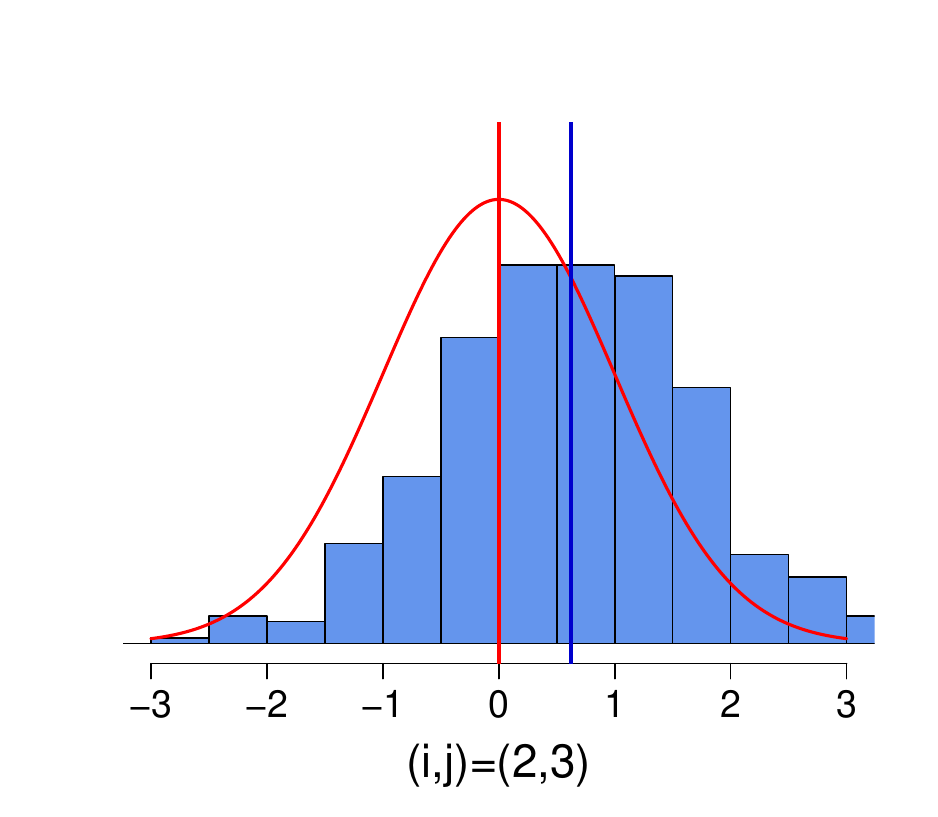}
    \end{minipage}
    \begin{minipage}{0.24\linewidth}
        \centering
        \includegraphics[width=\textwidth]{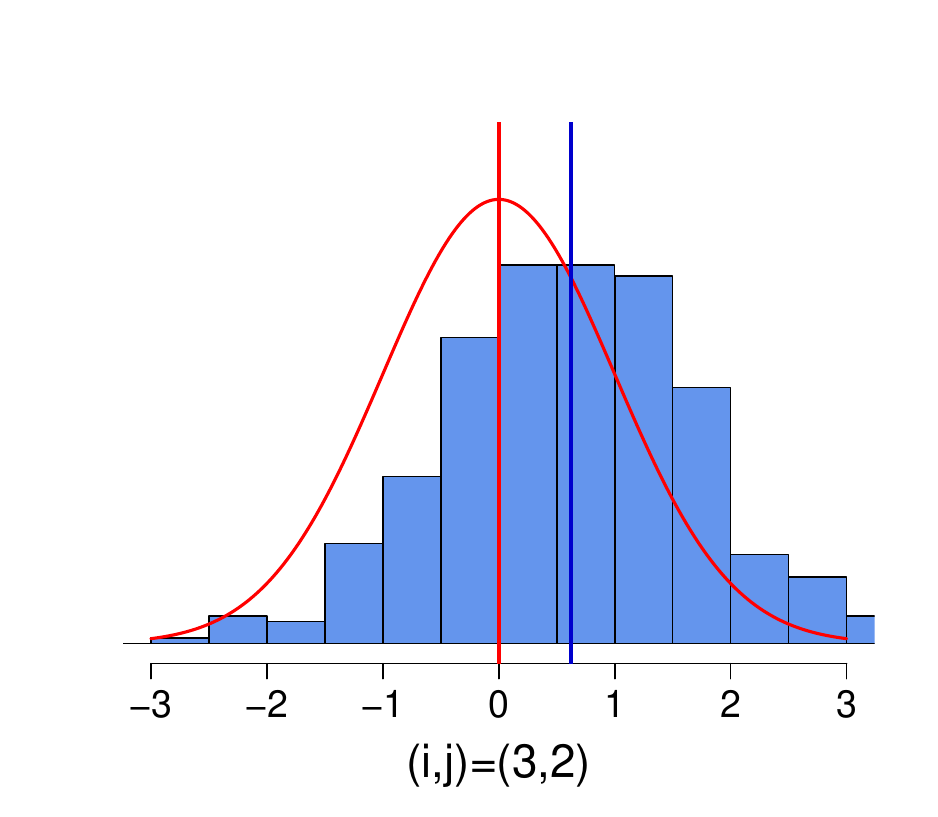}
    \end{minipage}
    \begin{minipage}{0.24\linewidth}
        \centering
        \includegraphics[width=\textwidth]{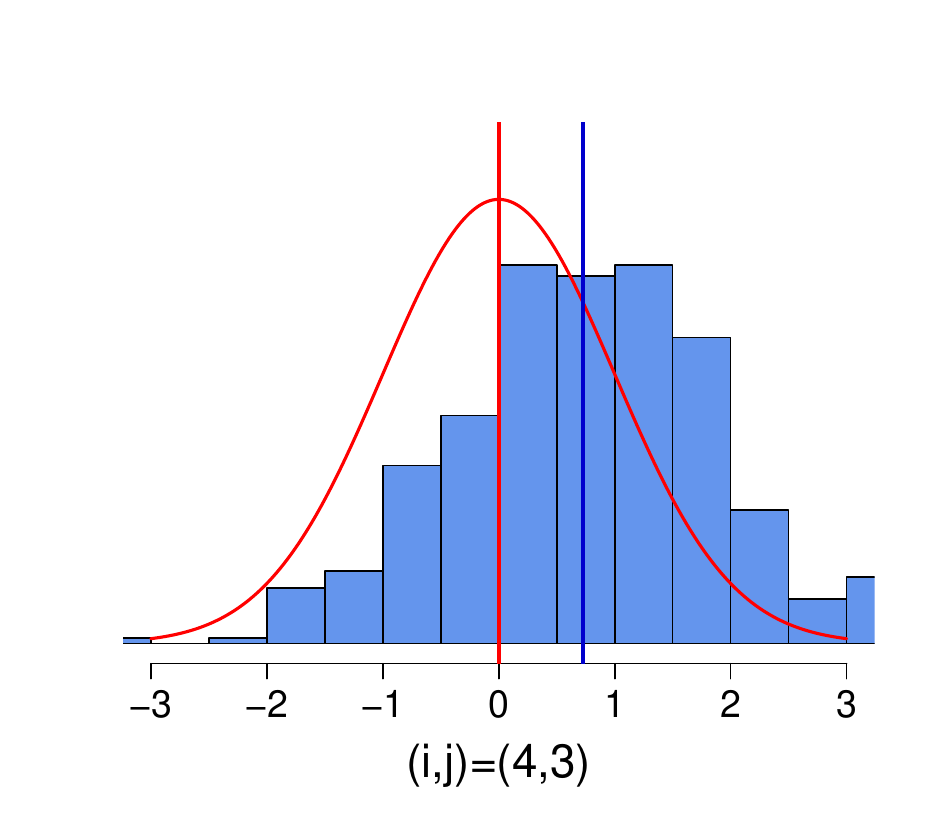}
    \end{minipage}    
 \end{minipage}
  \hspace{1cm}
 \begin{minipage}{0.3\linewidth}
     \begin{minipage}{0.24\linewidth}
        \centering
        \includegraphics[width=\textwidth]{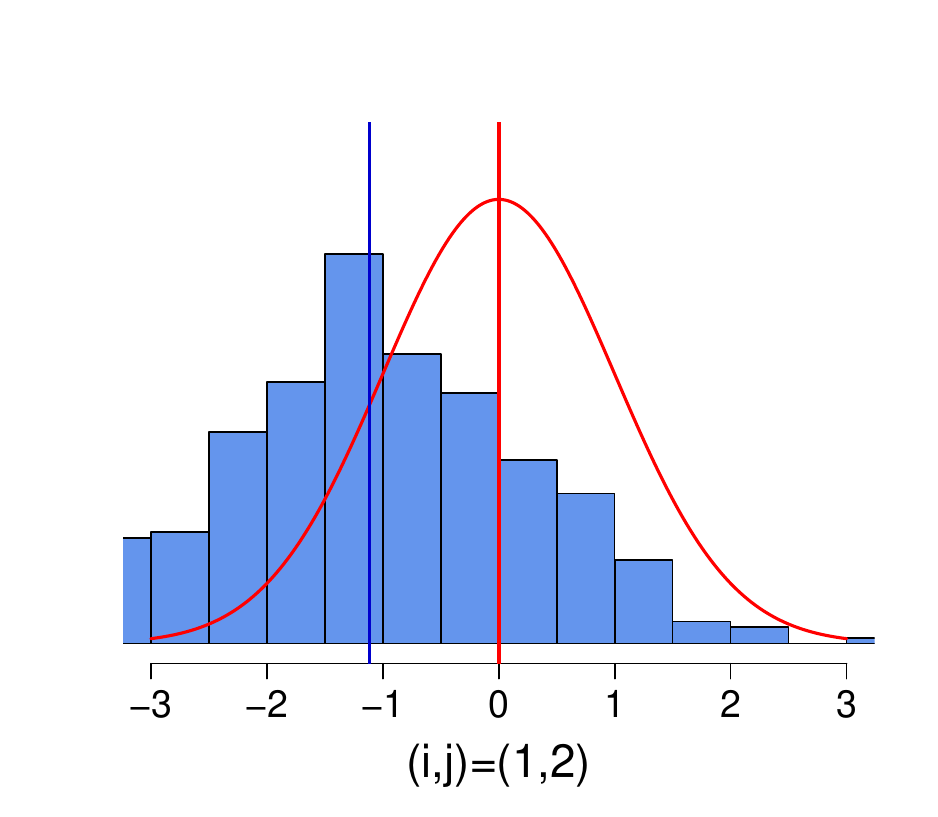}
    \end{minipage}
    \begin{minipage}{0.24\linewidth}
        \centering
        \includegraphics[width=\textwidth]{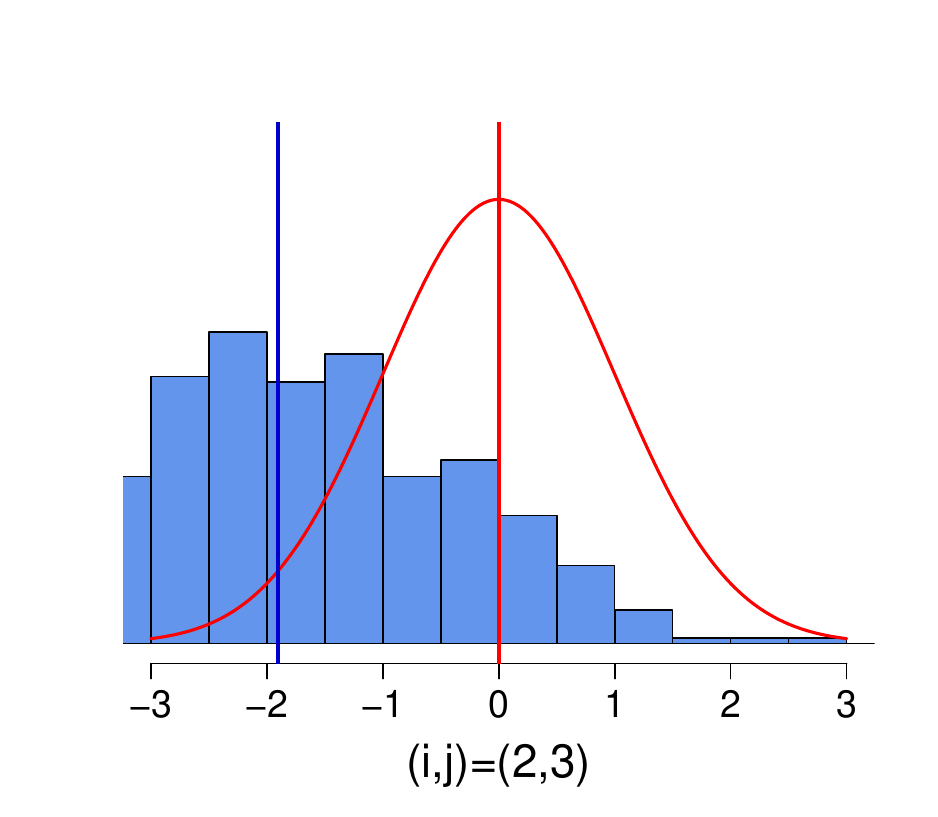}
    \end{minipage}
    \begin{minipage}{0.24\linewidth}
        \centering
        \includegraphics[width=\textwidth]{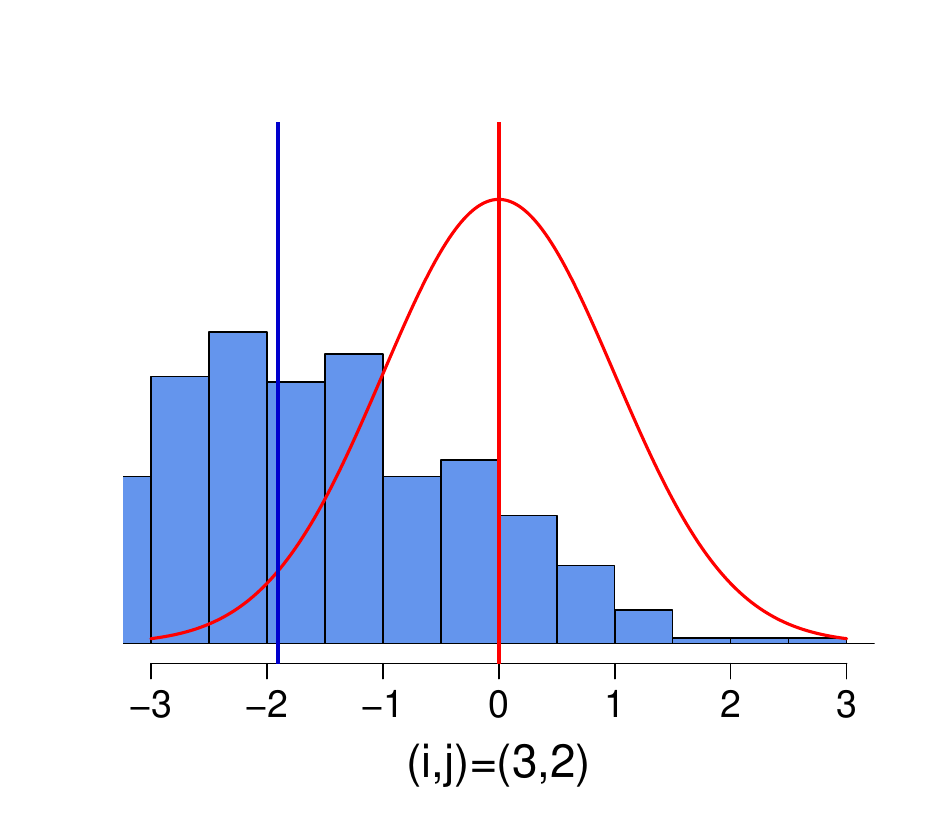}
    \end{minipage}
    \begin{minipage}{0.24\linewidth}
        \centering
        \includegraphics[width=\textwidth]{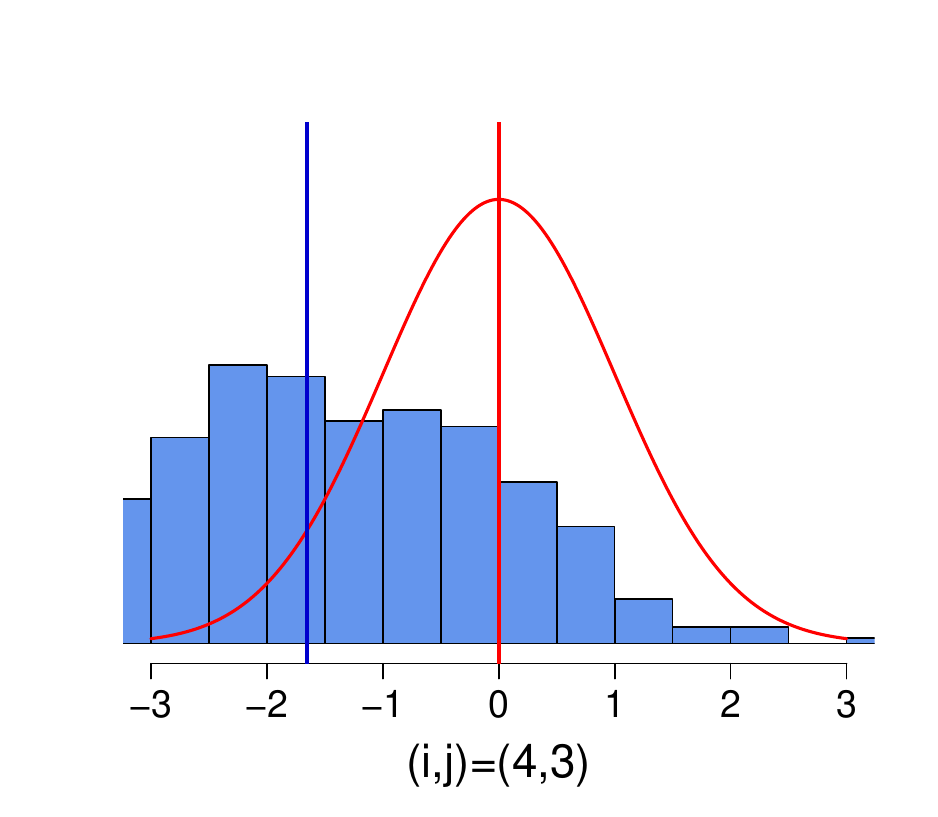}
    \end{minipage}
 \end{minipage}

  \caption*{$n=400, p=200$}
      \vspace{-0.43cm}
 \begin{minipage}{0.3\linewidth}
    \begin{minipage}{0.24\linewidth}
        \centering
        \includegraphics[width=\textwidth]{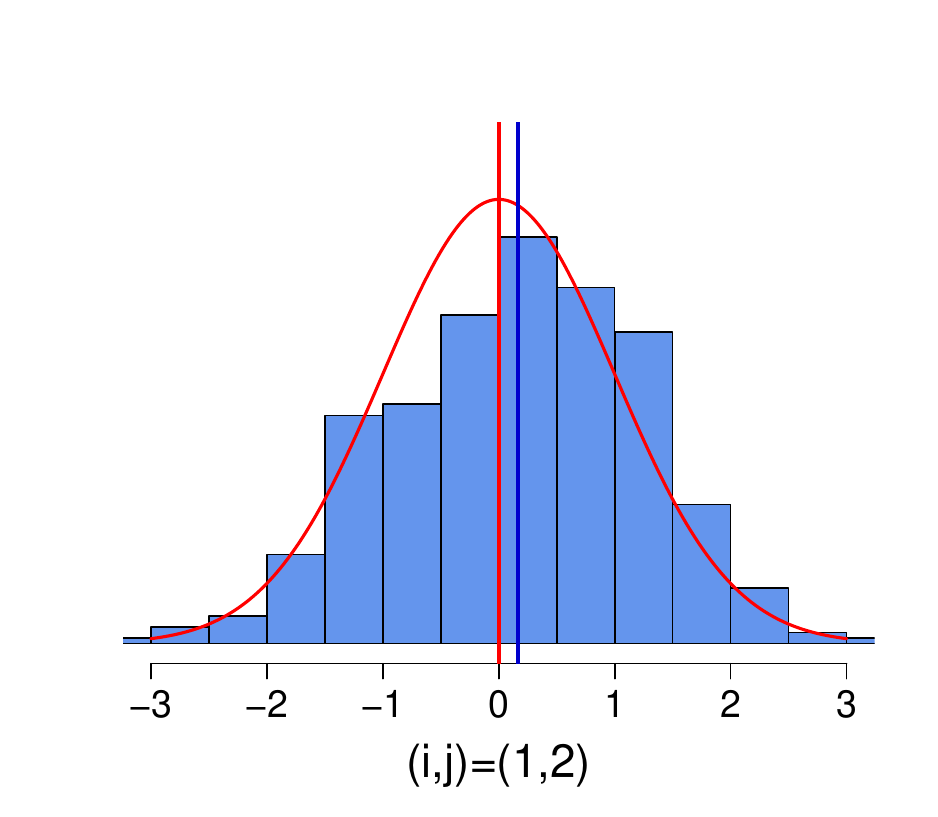}
    \end{minipage}
    \begin{minipage}{0.24\linewidth}
        \centering
        \includegraphics[width=\textwidth]{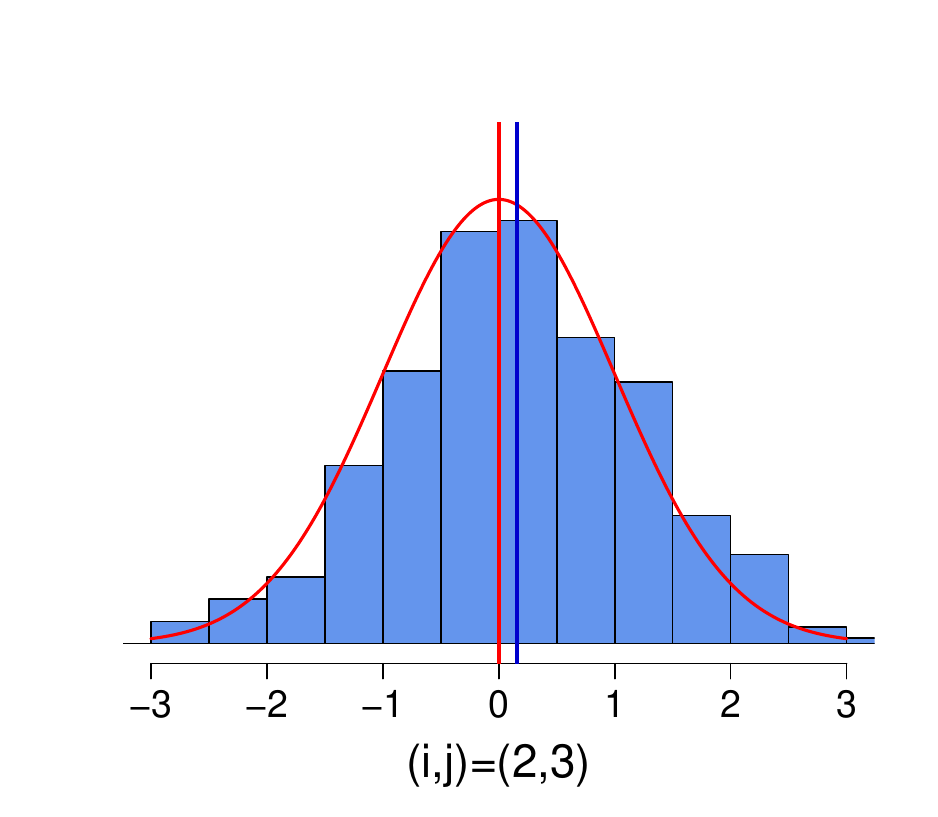}
    \end{minipage}
    \begin{minipage}{0.24\linewidth}
        \centering
        \includegraphics[width=\textwidth]{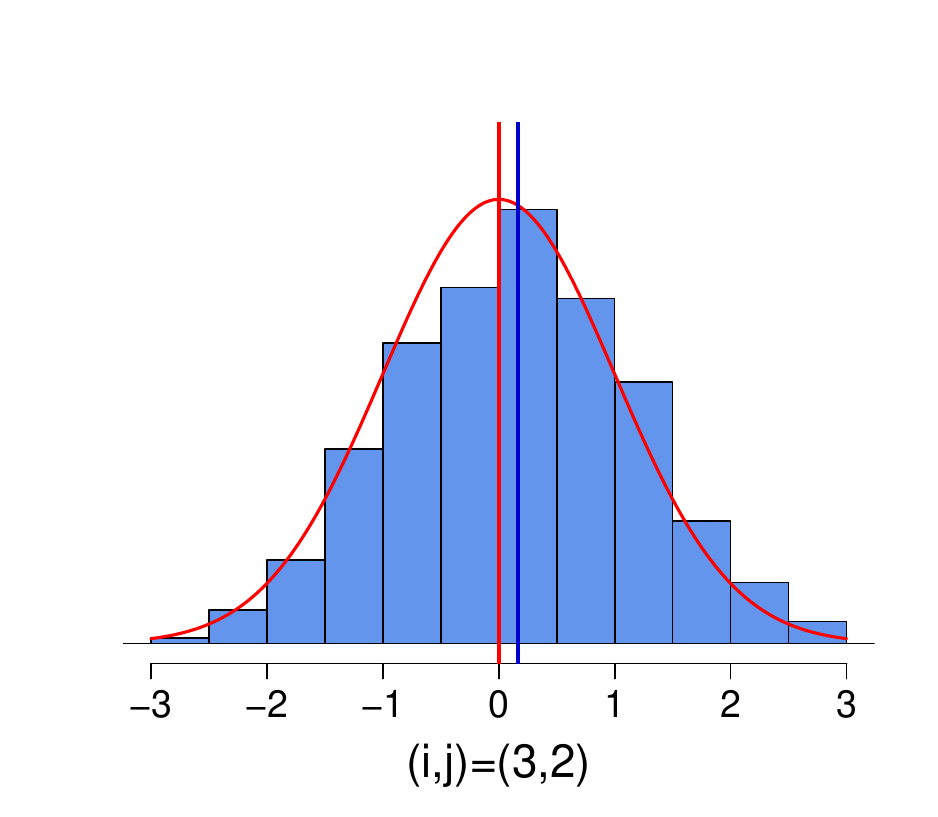}
    \end{minipage}
    \begin{minipage}{0.24\linewidth}
        \centering
        \includegraphics[width=\textwidth]{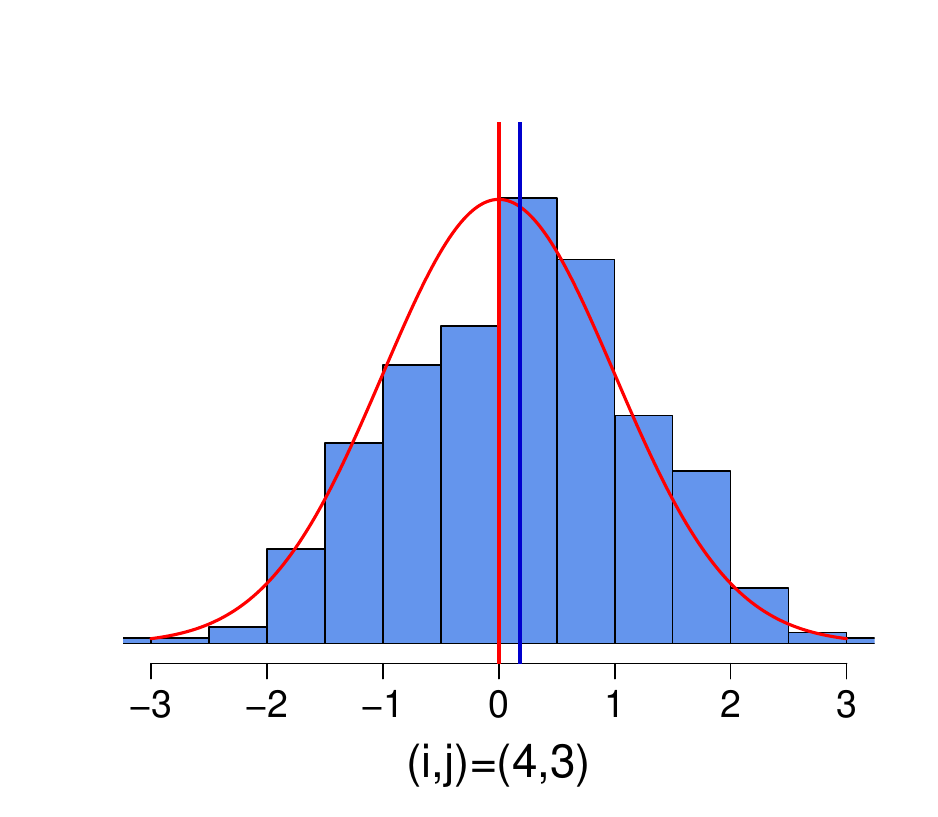}
    \end{minipage}
 \end{minipage}  
     \hspace{1cm}
 \begin{minipage}{0.3\linewidth}
    \begin{minipage}{0.24\linewidth}
        \centering
        \includegraphics[width=\textwidth]{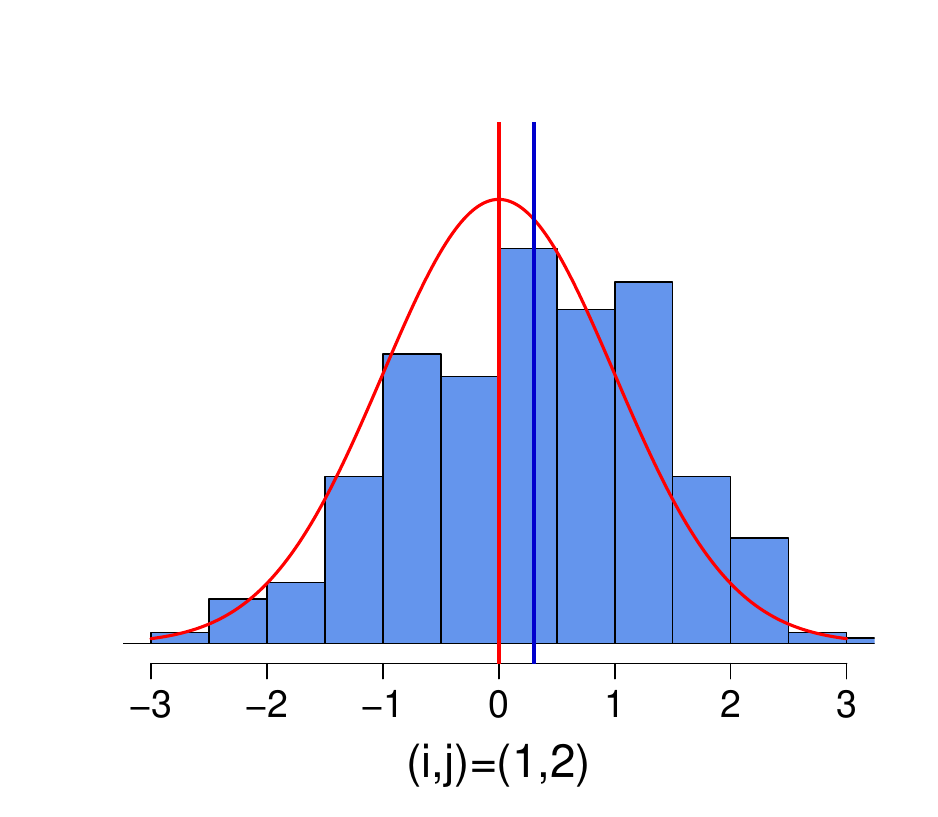}
    \end{minipage}
    \begin{minipage}{0.24\linewidth}
        \centering
        \includegraphics[width=\textwidth]{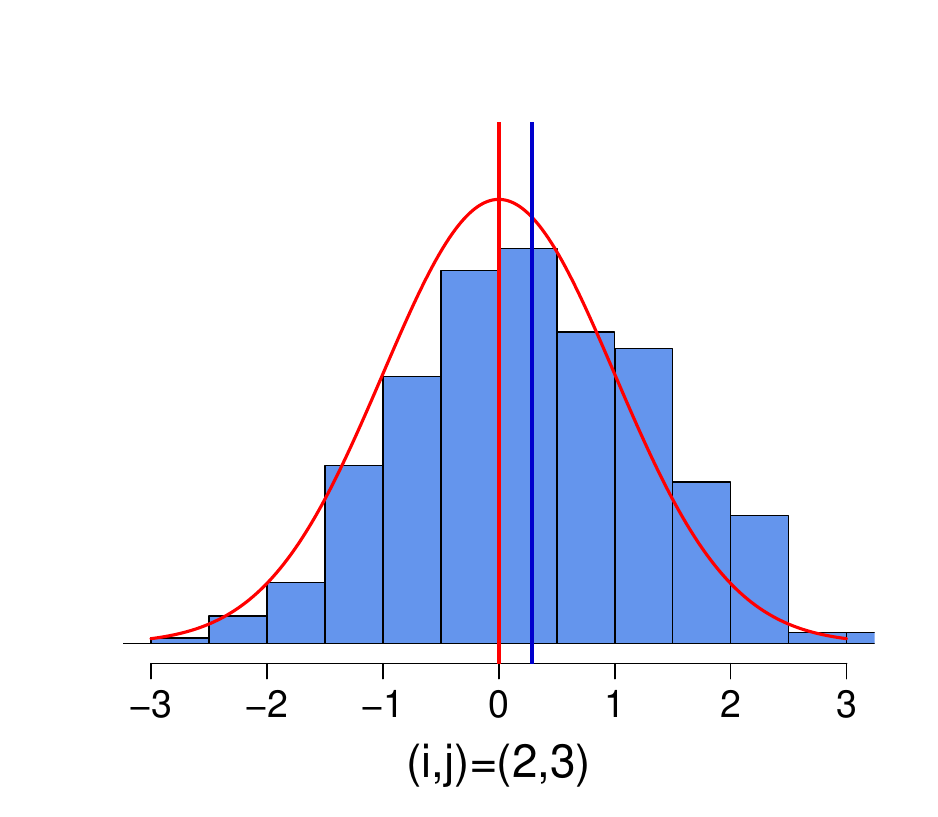}
    \end{minipage}
    \begin{minipage}{0.24\linewidth}
        \centering
        \includegraphics[width=\textwidth]{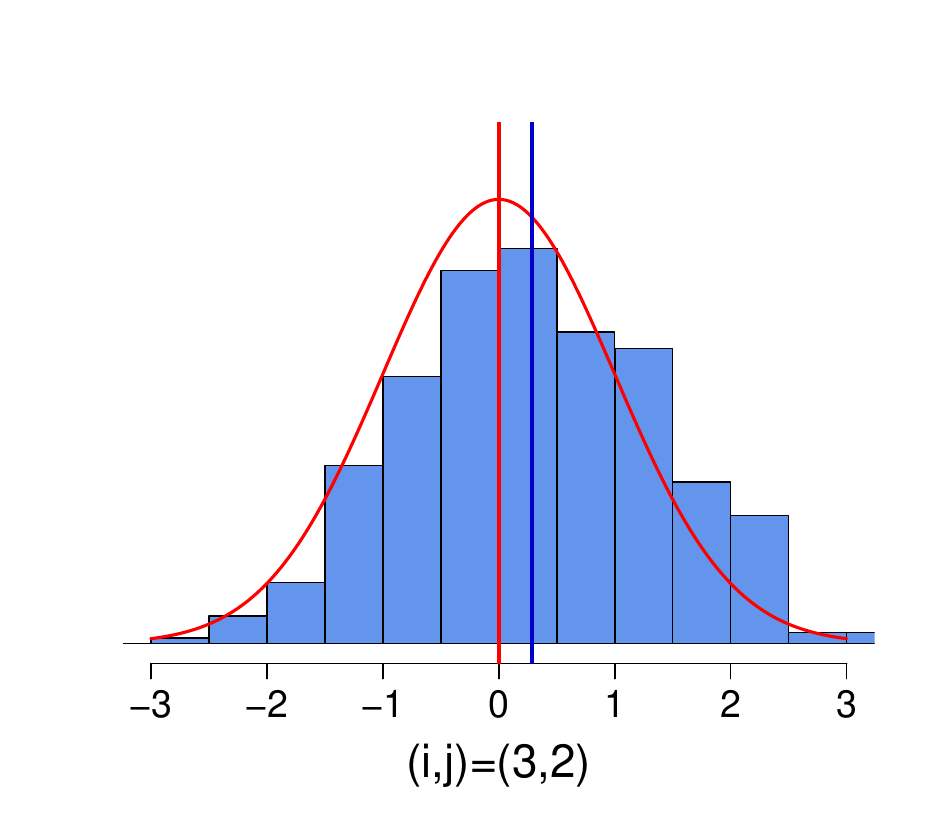}
    \end{minipage}
    \begin{minipage}{0.24\linewidth}
        \centering
        \includegraphics[width=\textwidth]{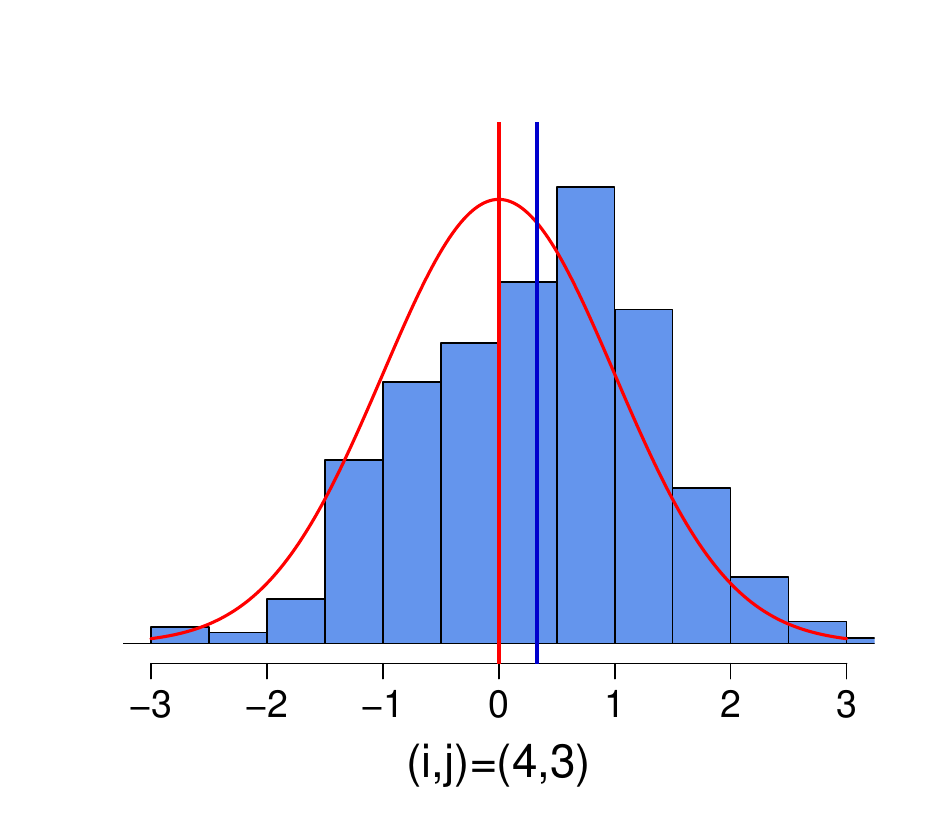}
    \end{minipage}
  \end{minipage}  
    \hspace{1cm}
 \begin{minipage}{0.3\linewidth}
    \begin{minipage}{0.24\linewidth}
        \centering
        \includegraphics[width=\textwidth]{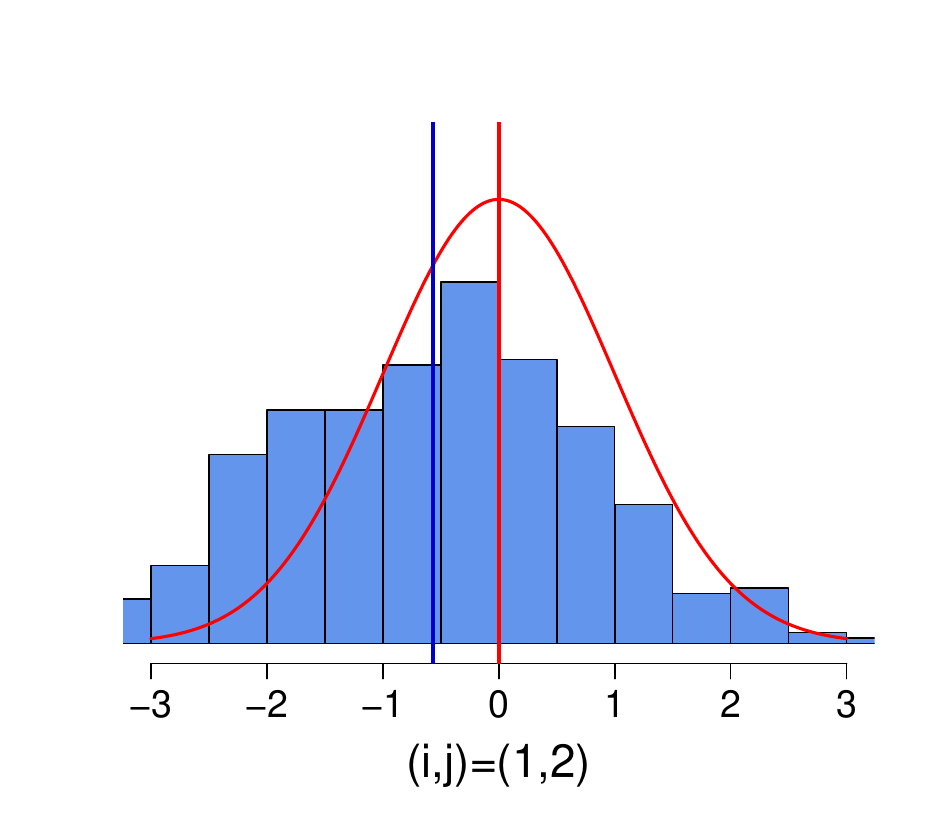}
    \end{minipage}
    \begin{minipage}{0.24\linewidth}
        \centering
        \includegraphics[width=\textwidth]{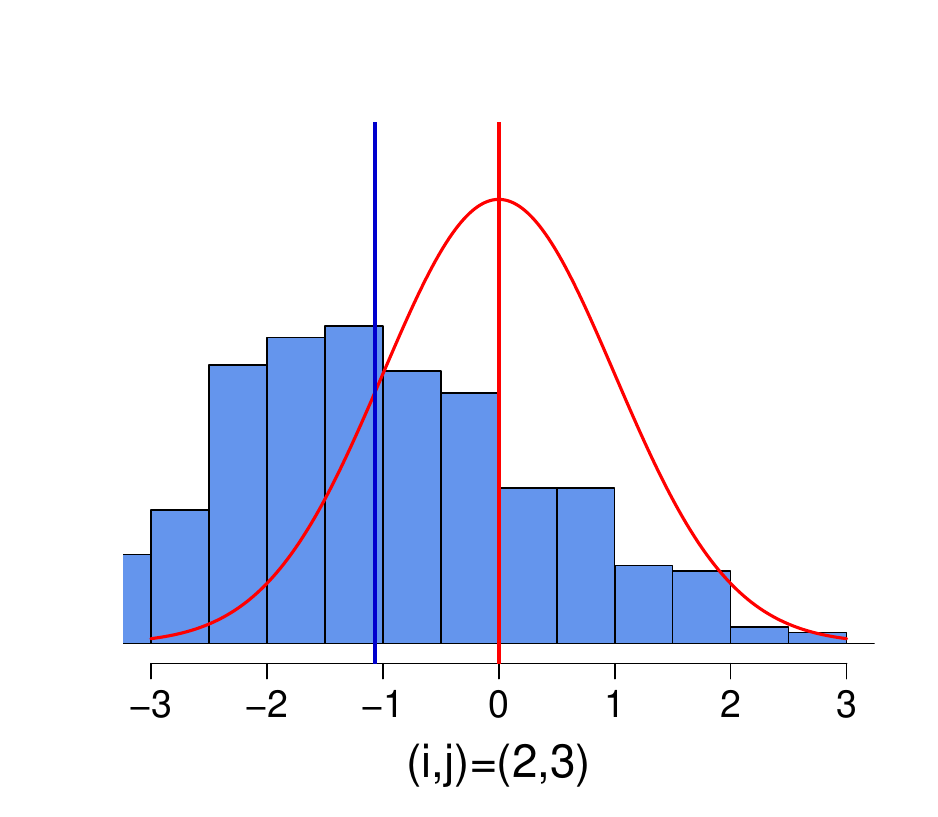}
    \end{minipage}
    \begin{minipage}{0.24\linewidth}
        \centering
        \includegraphics[width=\textwidth]{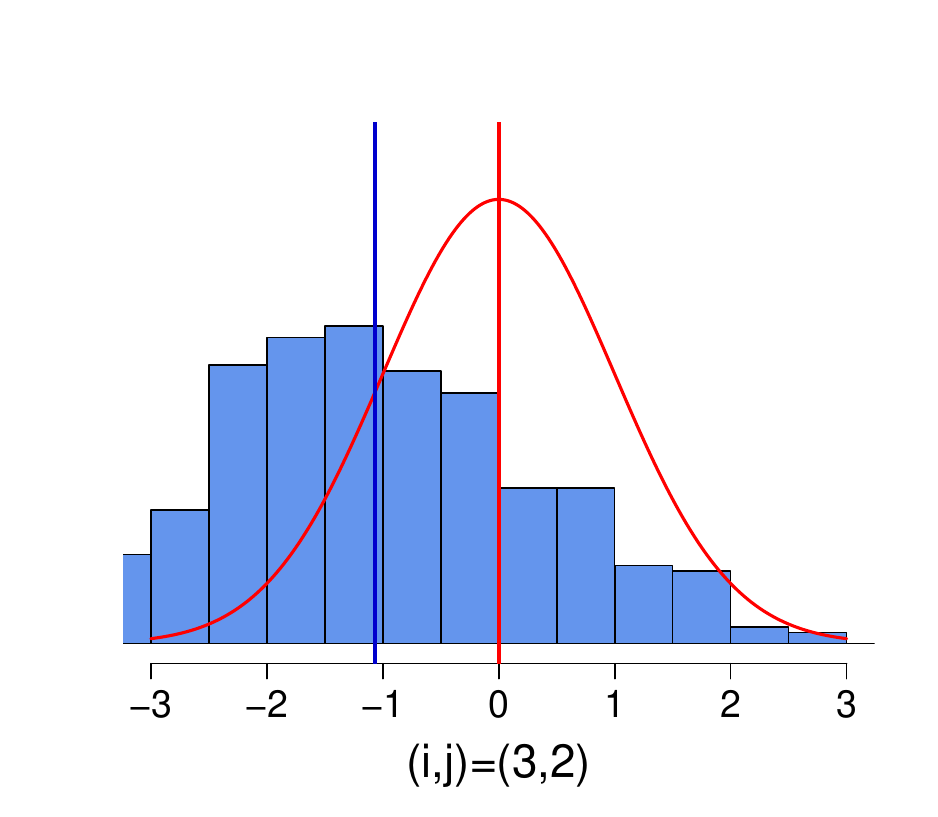}
    \end{minipage}
    \begin{minipage}{0.24\linewidth}
        \centering
        \includegraphics[width=\textwidth]{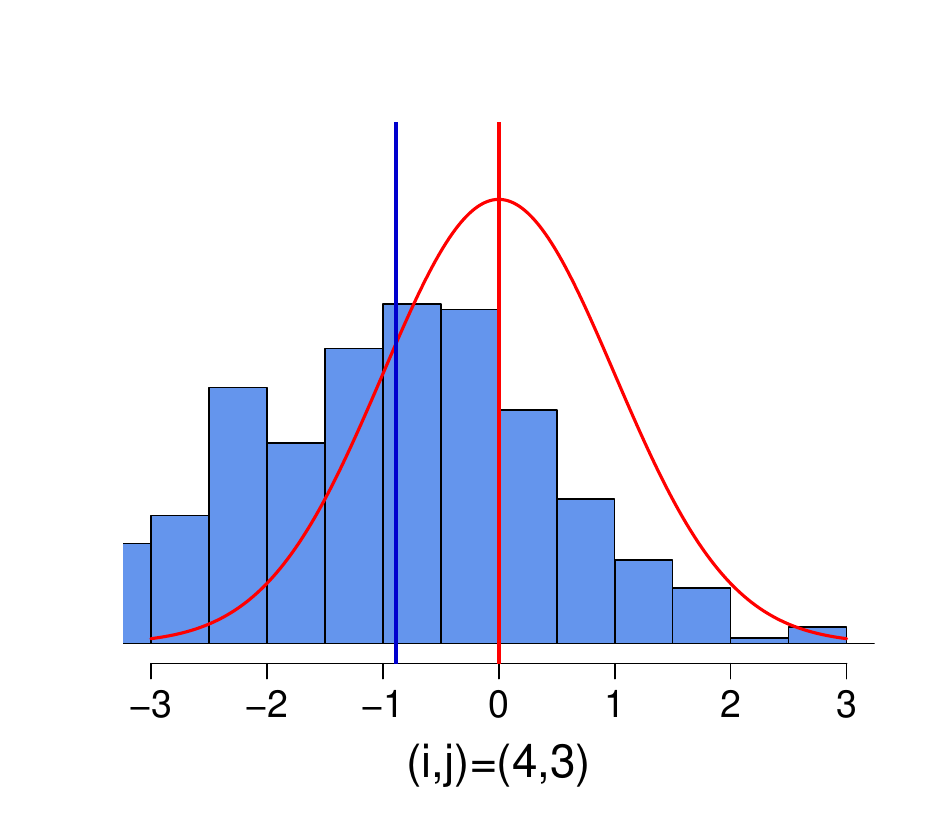}
    \end{minipage}
 \end{minipage}

  \caption*{$n=800, p=200$}
      \vspace{-0.43cm}
 \begin{minipage}{0.3\linewidth}
    \begin{minipage}{0.24\linewidth}
        \centering
        \includegraphics[width=\textwidth]{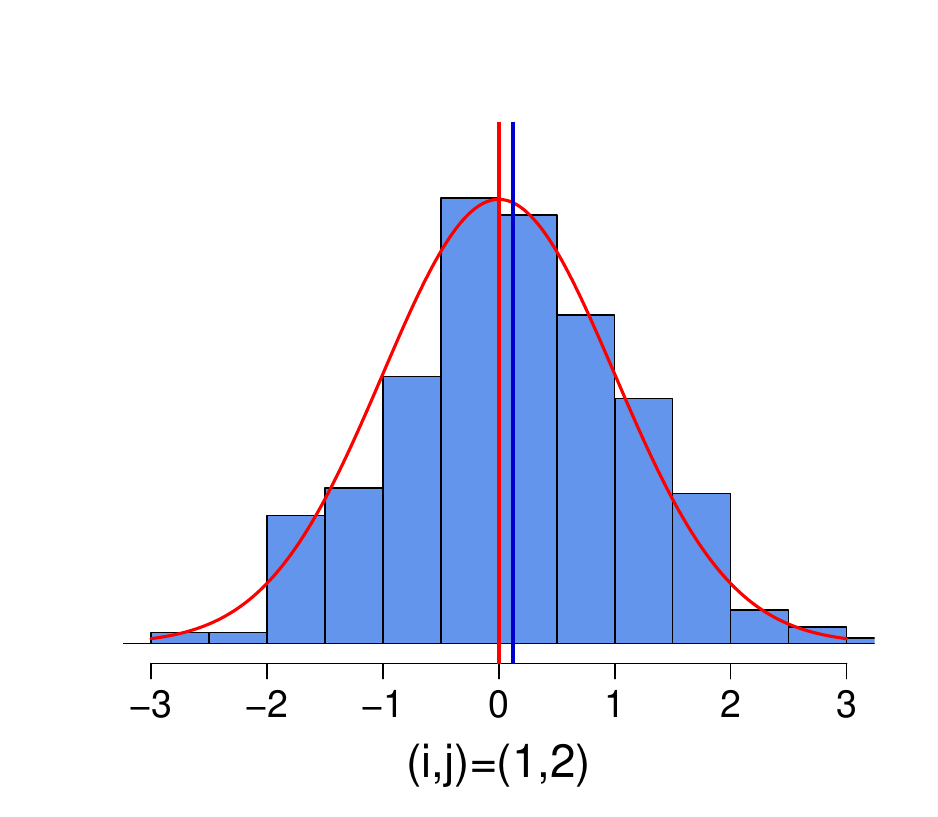}
    \end{minipage}
    \begin{minipage}{0.24\linewidth}
        \centering
        \includegraphics[width=\textwidth]{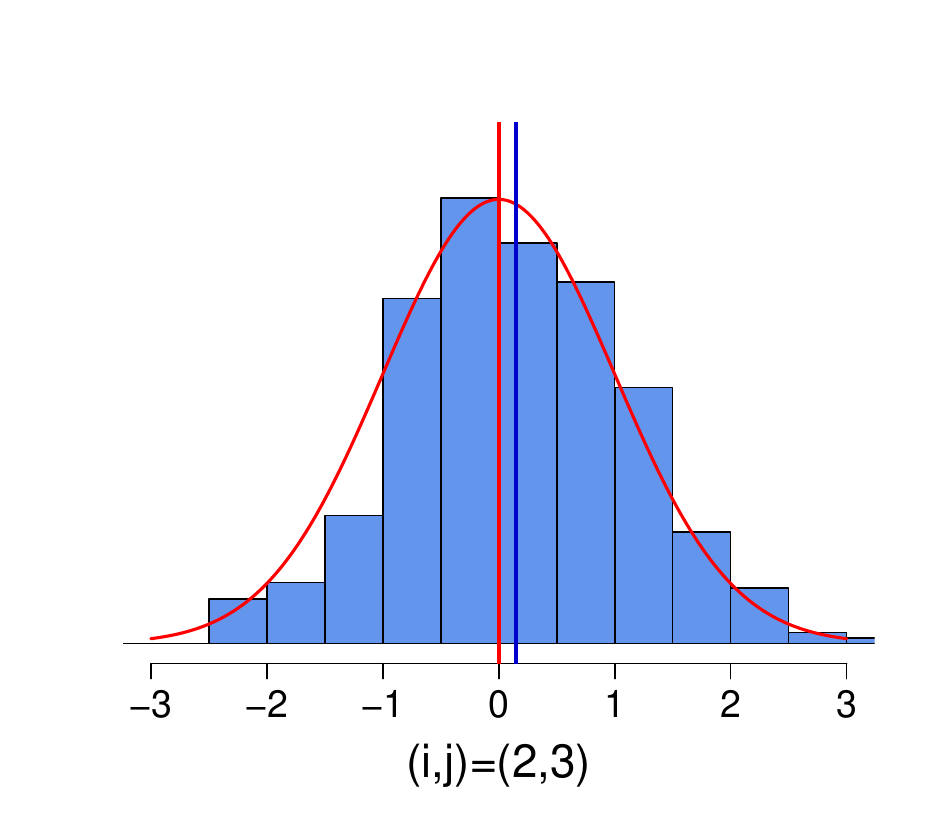}
    \end{minipage}
    \begin{minipage}{0.24\linewidth}
        \centering
        \includegraphics[width=\textwidth]{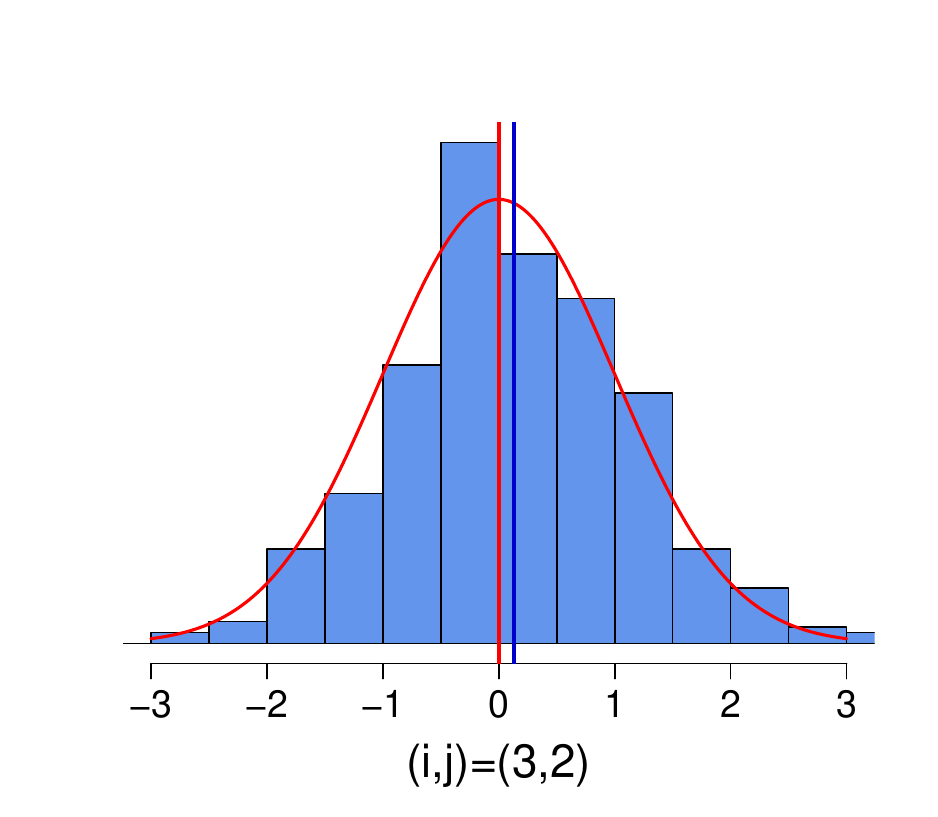}
    \end{minipage}
    \begin{minipage}{0.24\linewidth}
        \centering
        \includegraphics[width=\textwidth]{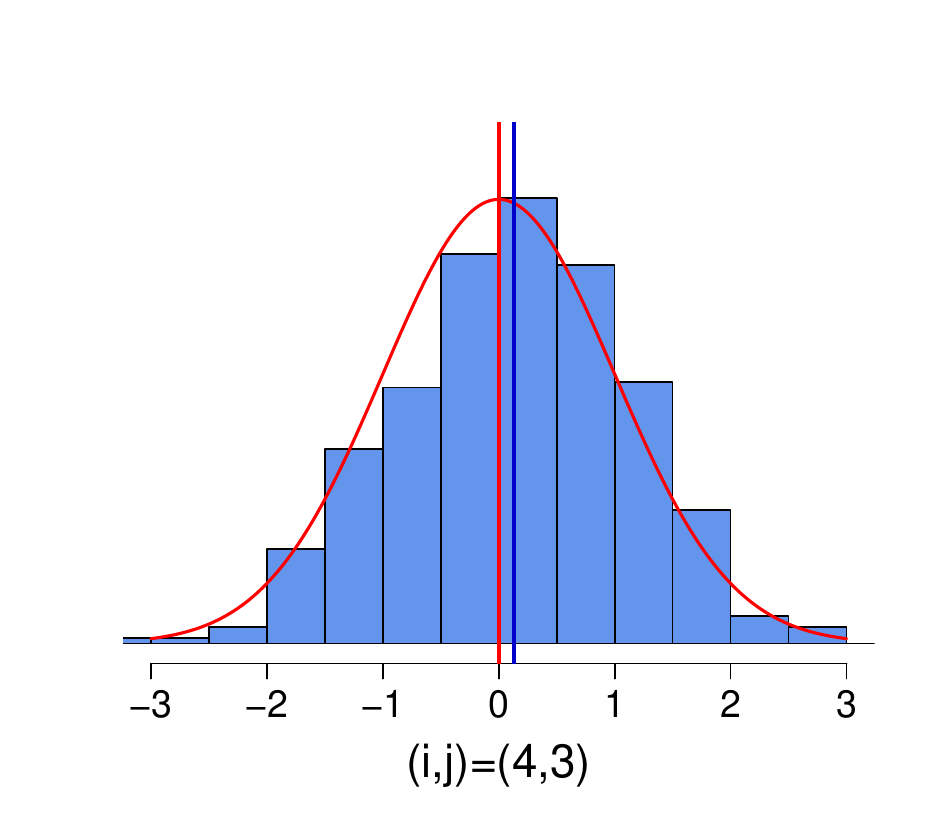}
    \end{minipage}
 \end{minipage} 
     \hspace{1cm}
 \begin{minipage}{0.3\linewidth}
    \begin{minipage}{0.24\linewidth}
        \centering
        \includegraphics[width=\textwidth]{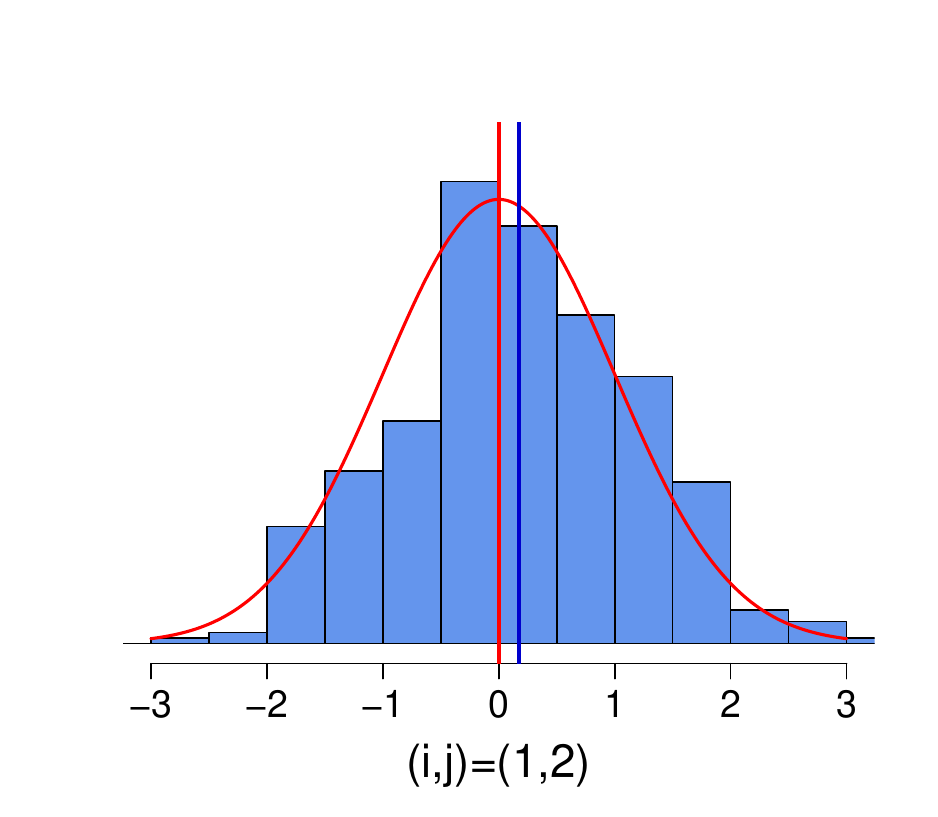}
    \end{minipage}
    \begin{minipage}{0.24\linewidth}
        \centering
        \includegraphics[width=\textwidth]{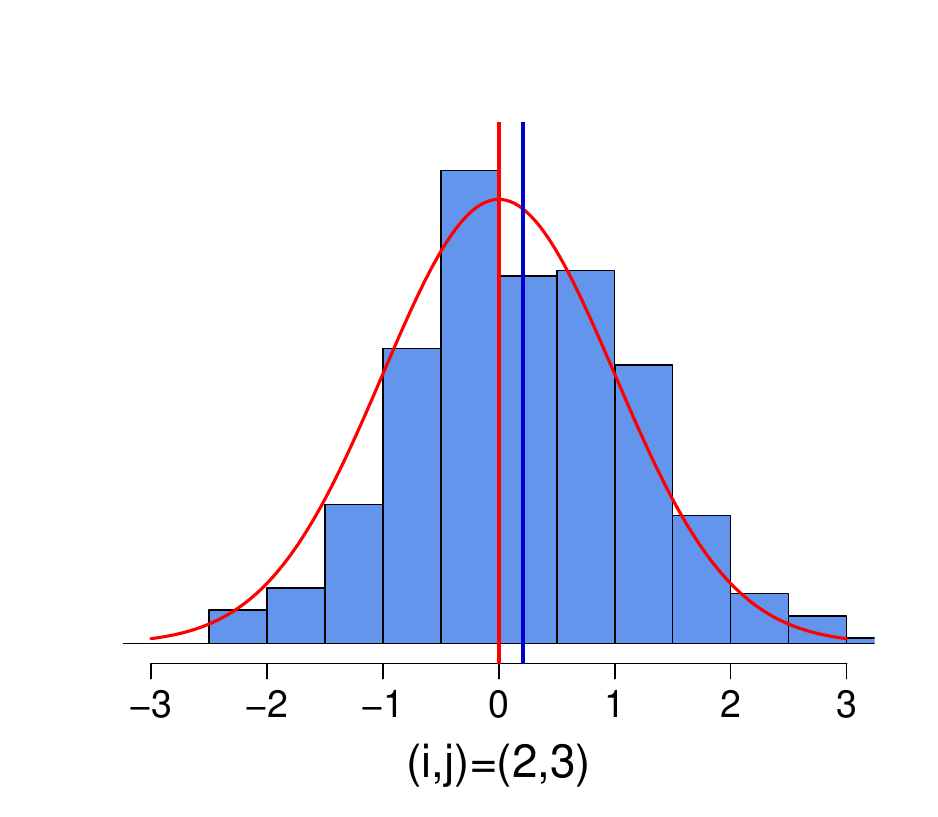}
    \end{minipage}
    \begin{minipage}{0.24\linewidth}
        \centering
        \includegraphics[width=\textwidth]{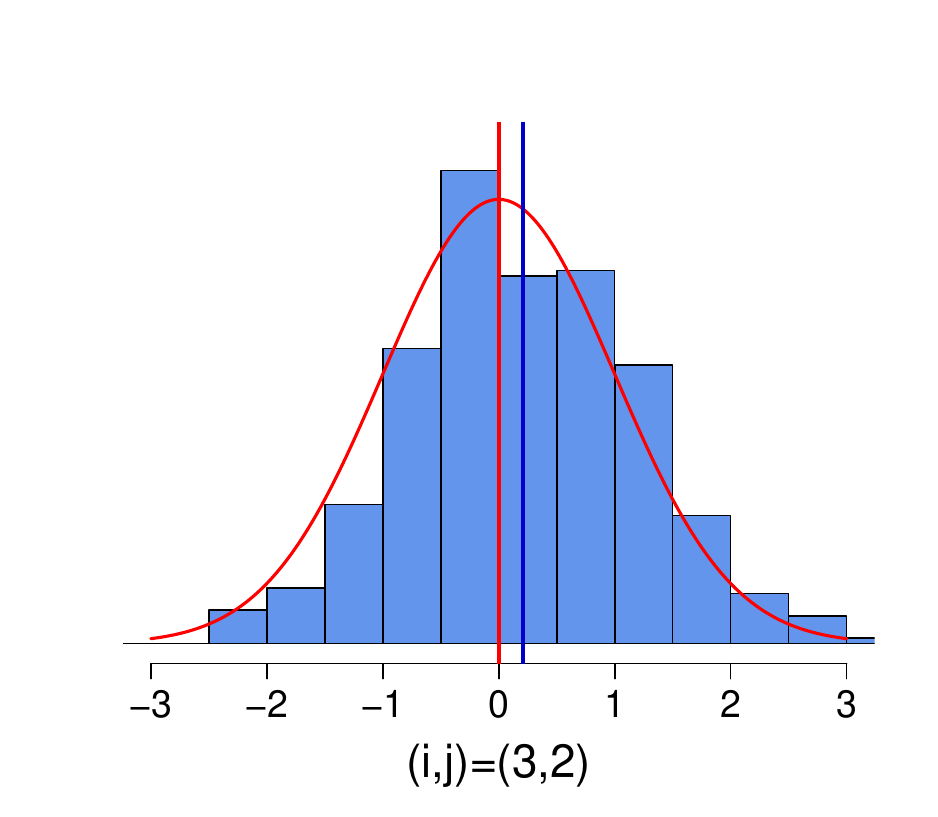}
    \end{minipage}
    \begin{minipage}{0.24\linewidth}
        \centering
        \includegraphics[width=\textwidth]{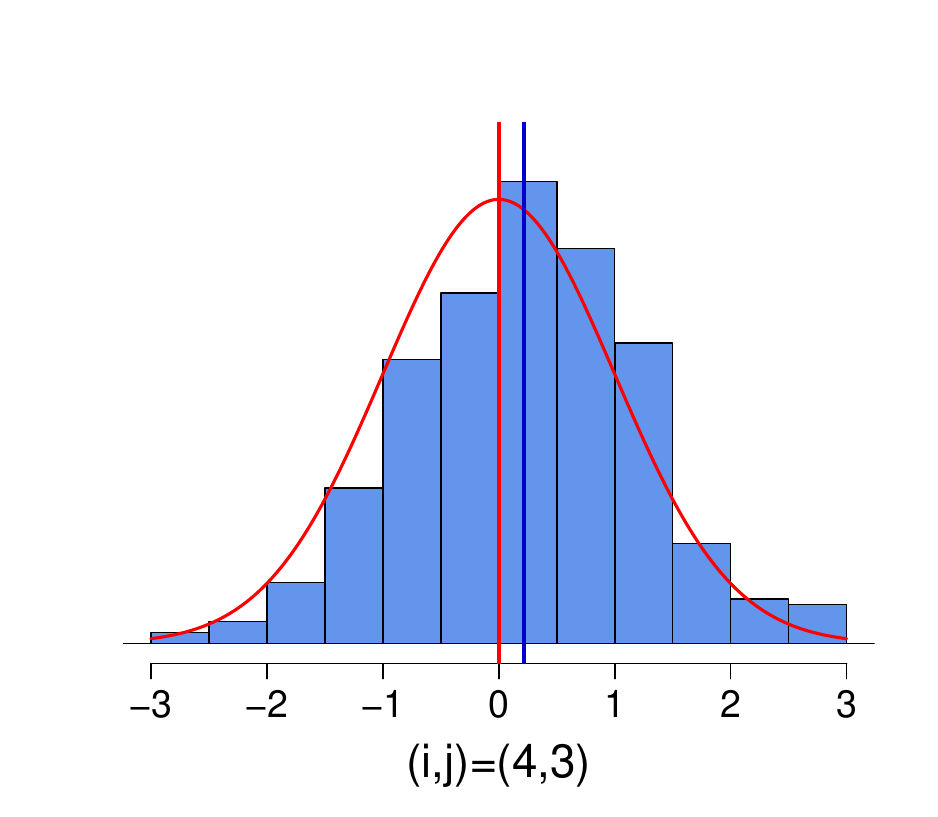}
    \end{minipage}
 \end{minipage}   
      \hspace{1cm}
 \begin{minipage}{0.3\linewidth}
    \begin{minipage}{0.24\linewidth}
        \centering
        \includegraphics[width=\textwidth]{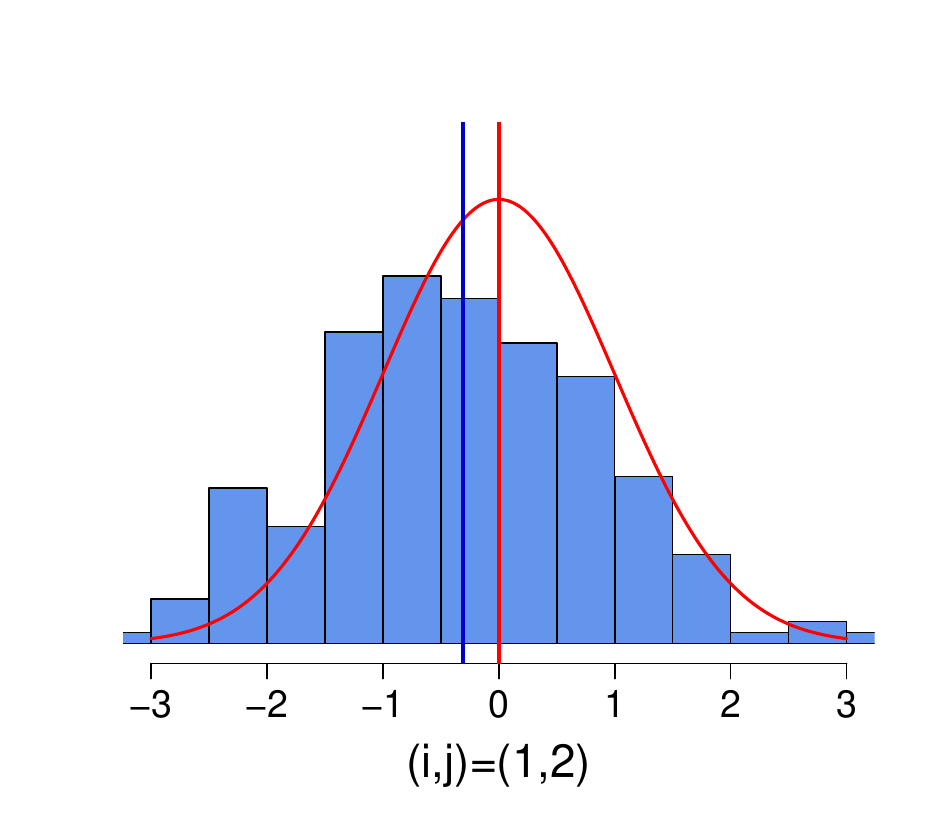}
    \end{minipage}
    \begin{minipage}{0.24\linewidth}
        \centering
        \includegraphics[width=\textwidth]{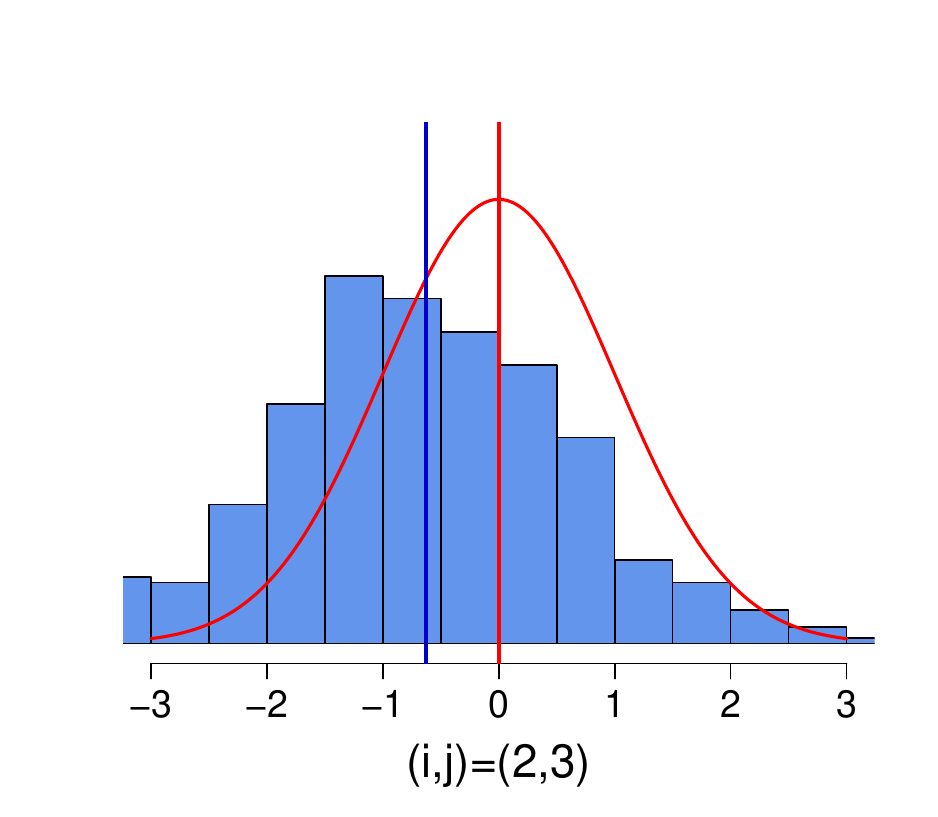}
    \end{minipage}
    \begin{minipage}{0.24\linewidth}
        \centering
        \includegraphics[width=\textwidth]{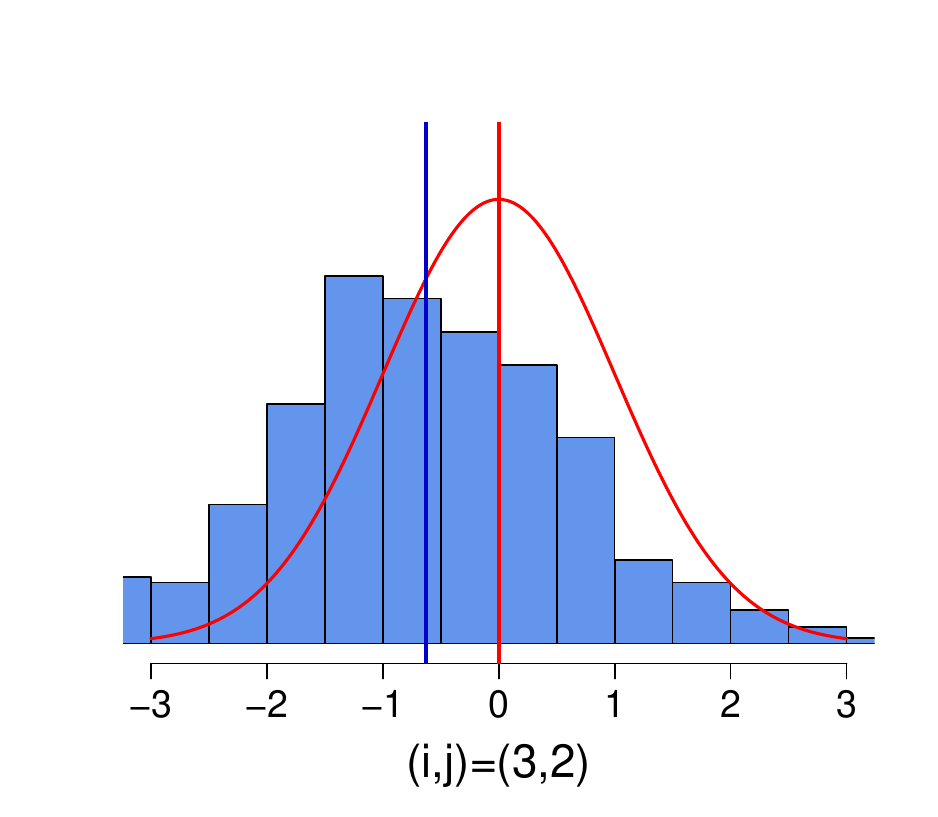}
    \end{minipage}
    \begin{minipage}{0.24\linewidth}
        \centering
        \includegraphics[width=\textwidth]{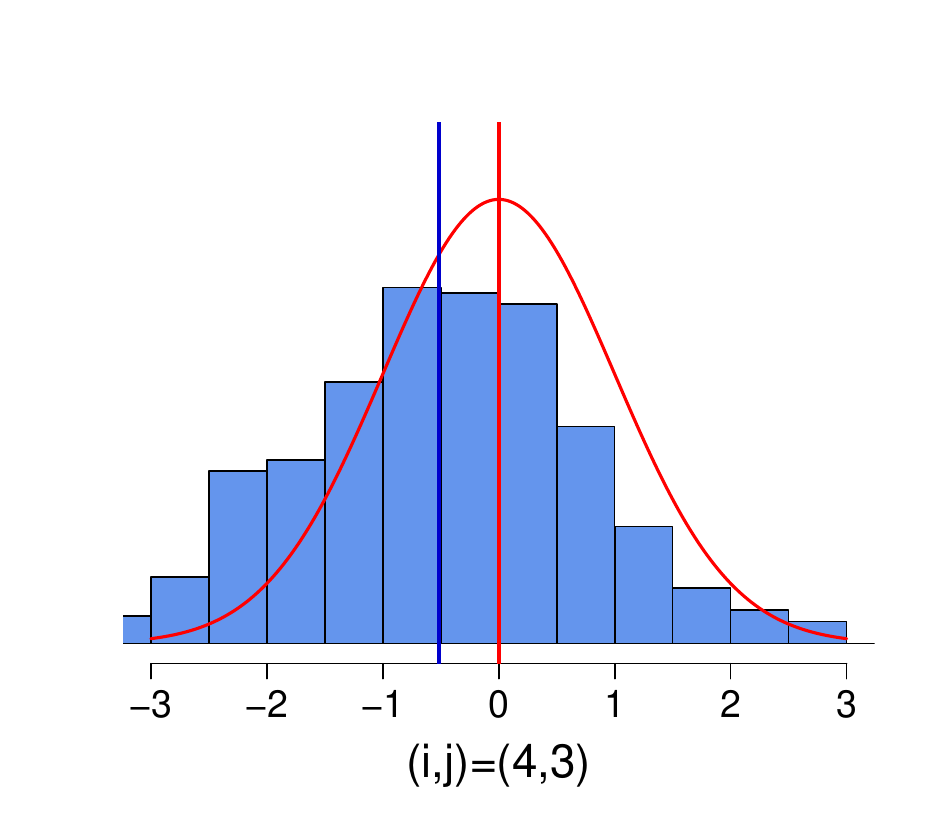}
    \end{minipage}
     \end{minipage}   
     
 \caption*{$n=200, p=400$}
     \vspace{-0.43cm}
 \begin{minipage}{0.3\linewidth}
    \begin{minipage}{0.24\linewidth}
        \centering
        \includegraphics[width=\textwidth]{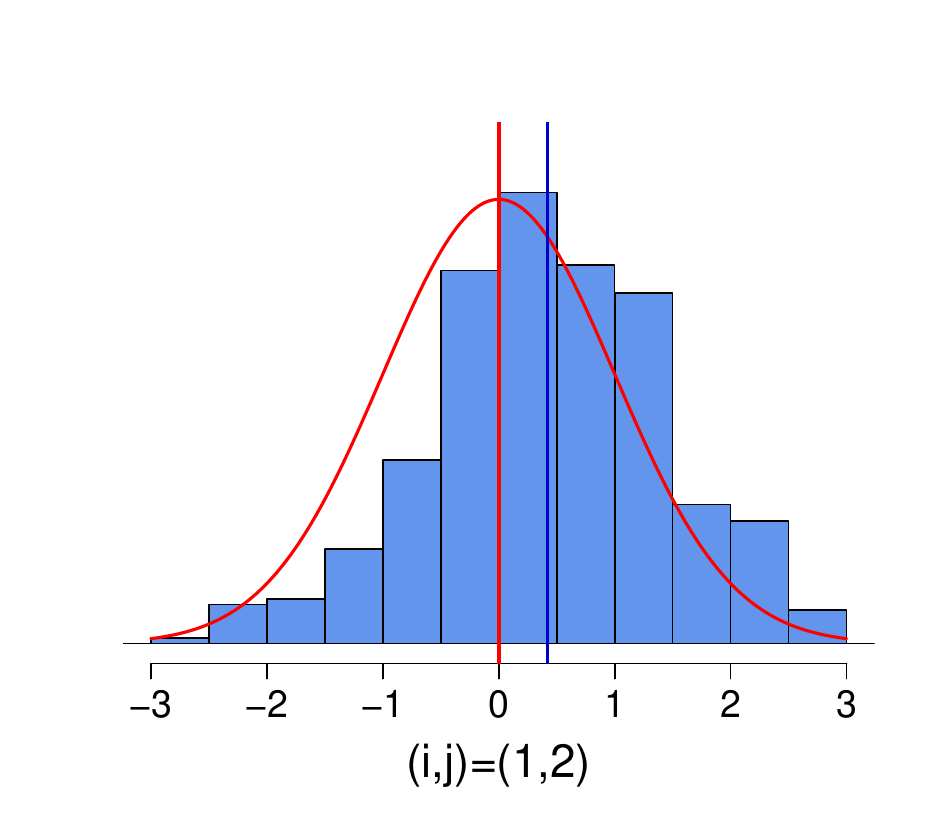}
    \end{minipage}
    \begin{minipage}{0.24\linewidth}
        \centering
        \includegraphics[width=\textwidth]{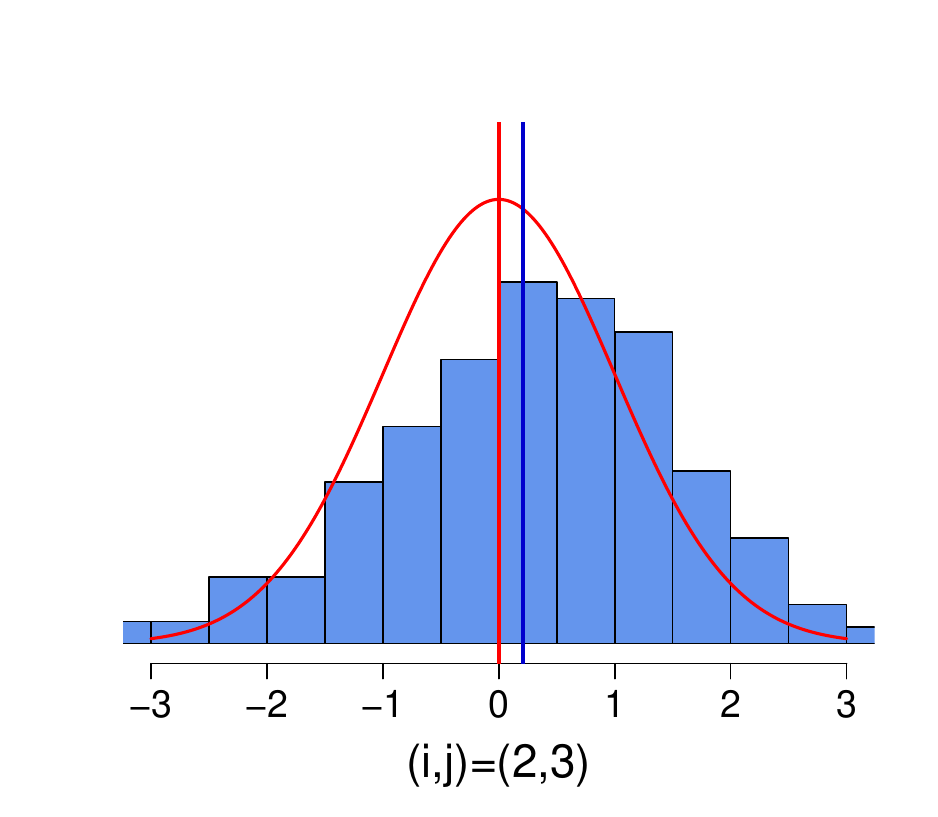}
    \end{minipage}
    \begin{minipage}{0.24\linewidth}
        \centering
        \includegraphics[width=\textwidth]{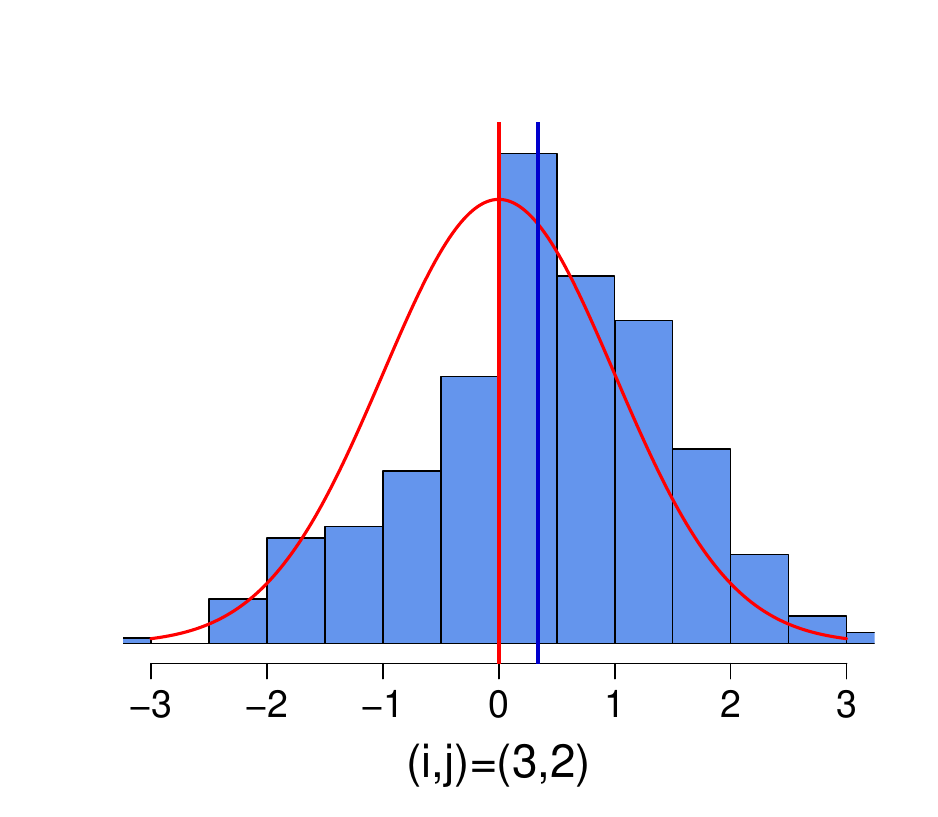}
    \end{minipage}
    \begin{minipage}{0.24\linewidth}
        \centering
        \includegraphics[width=\textwidth]{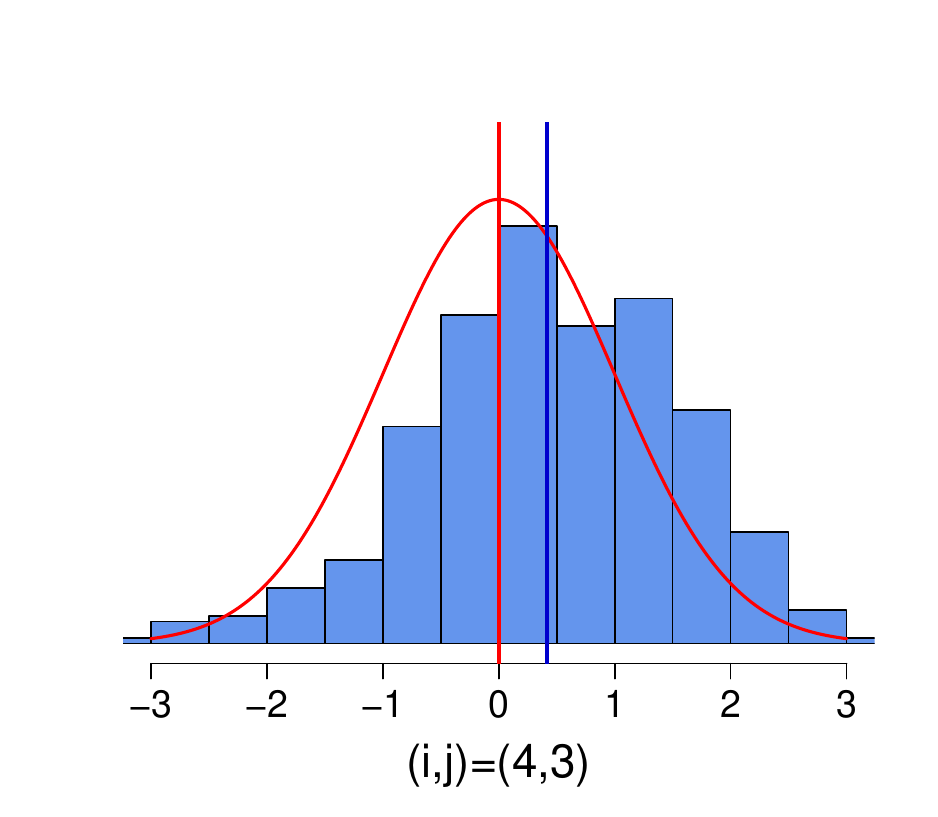}
    \end{minipage}
 \end{minipage}
 \hspace{1cm}
 \begin{minipage}{0.3\linewidth}
    \begin{minipage}{0.24\linewidth}
        \centering
        \includegraphics[width=\textwidth]{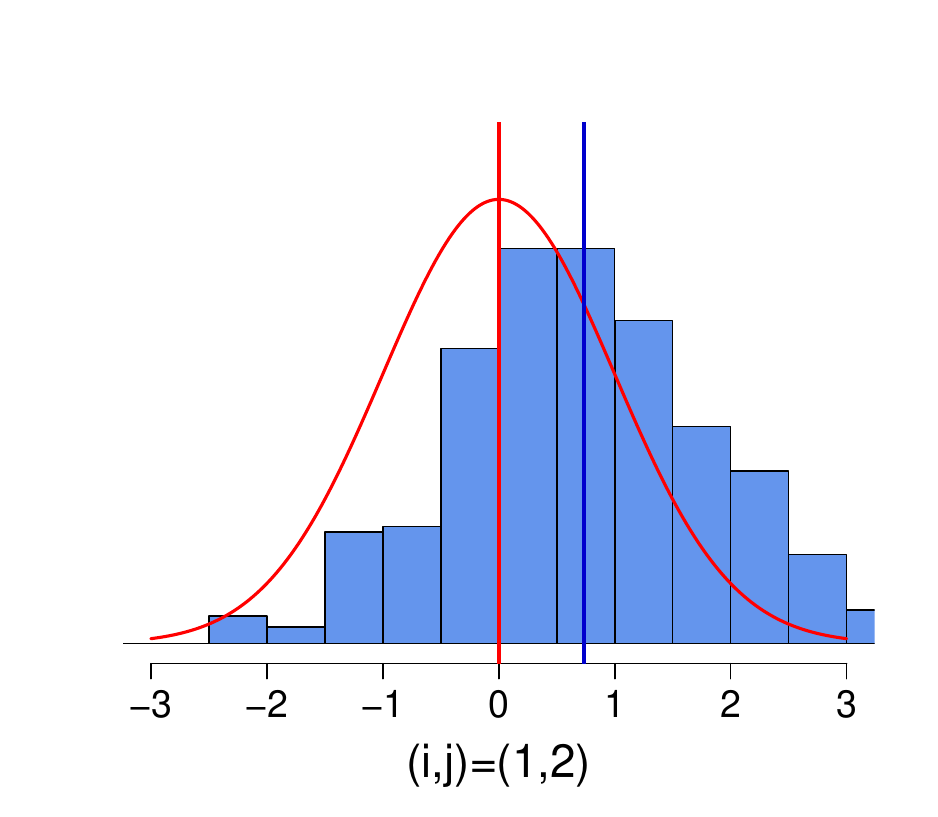}
    \end{minipage}
    \begin{minipage}{0.24\linewidth}
        \centering
        \includegraphics[width=\textwidth]{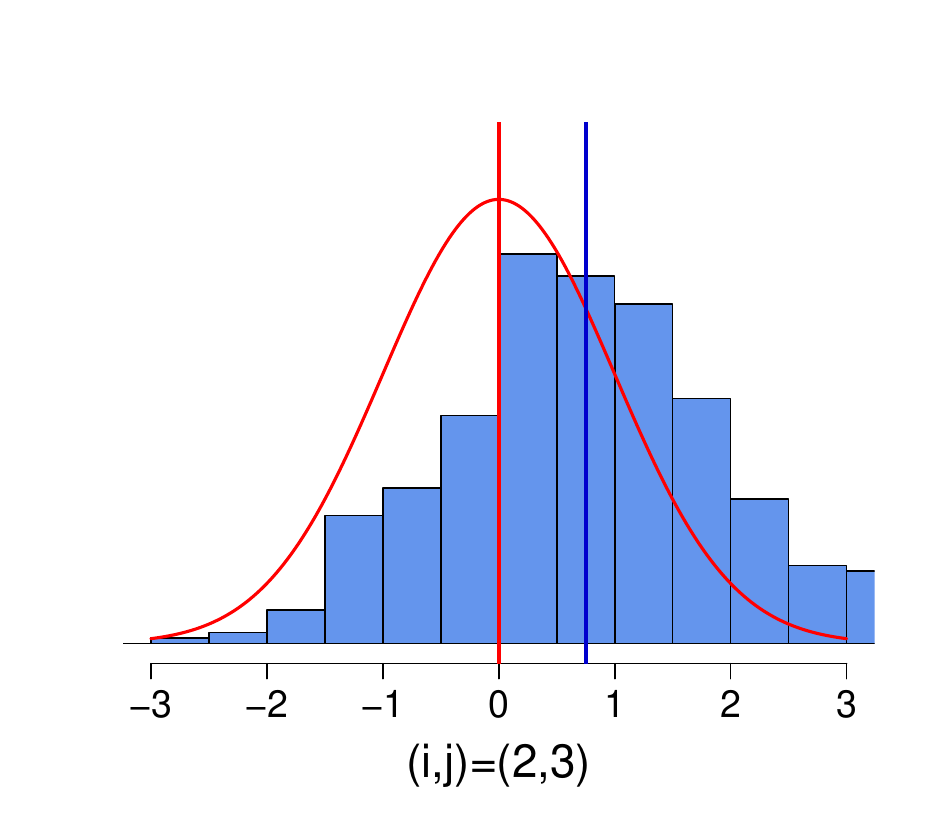}
    \end{minipage}
    \begin{minipage}{0.24\linewidth}
        \centering
        \includegraphics[width=\textwidth]{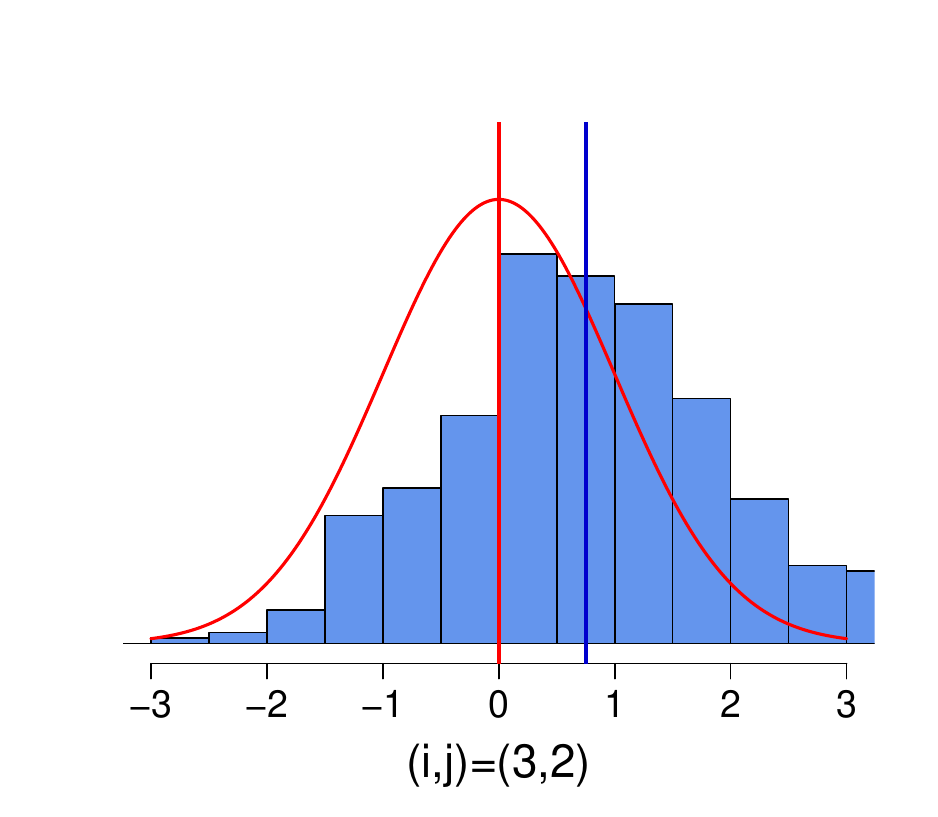}
    \end{minipage}
    \begin{minipage}{0.24\linewidth}
        \centering
        \includegraphics[width=\textwidth]{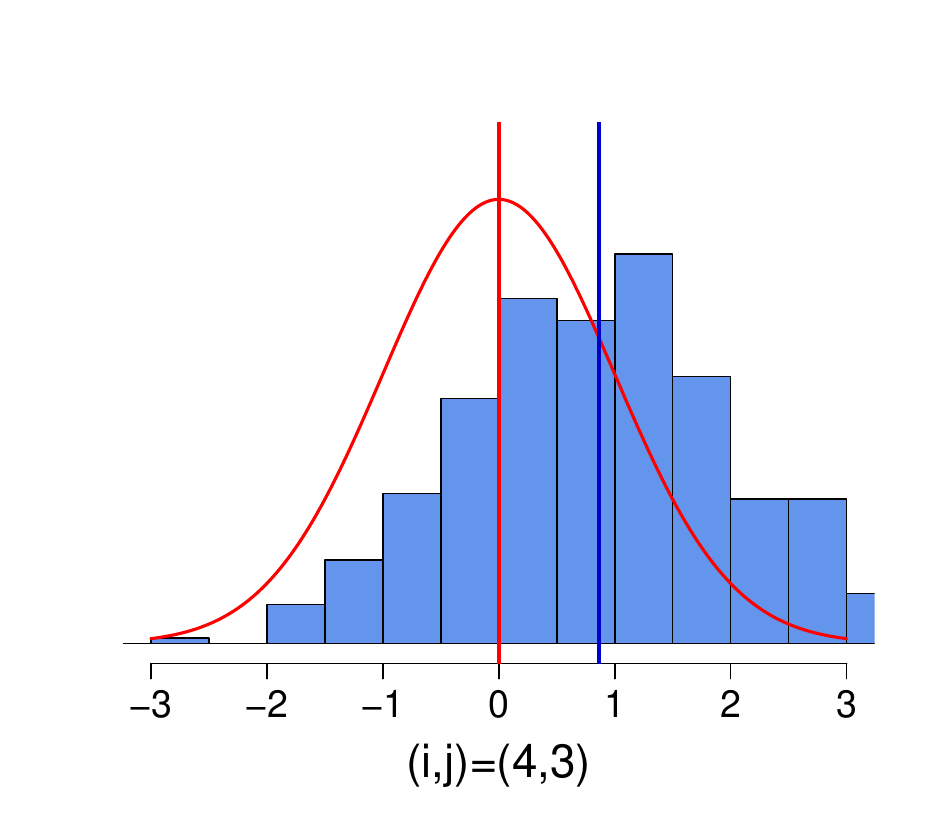}
    \end{minipage}    
 \end{minipage}
  \hspace{1cm}
 \begin{minipage}{0.3\linewidth}
     \begin{minipage}{0.24\linewidth}
        \centering
        \includegraphics[width=\textwidth]{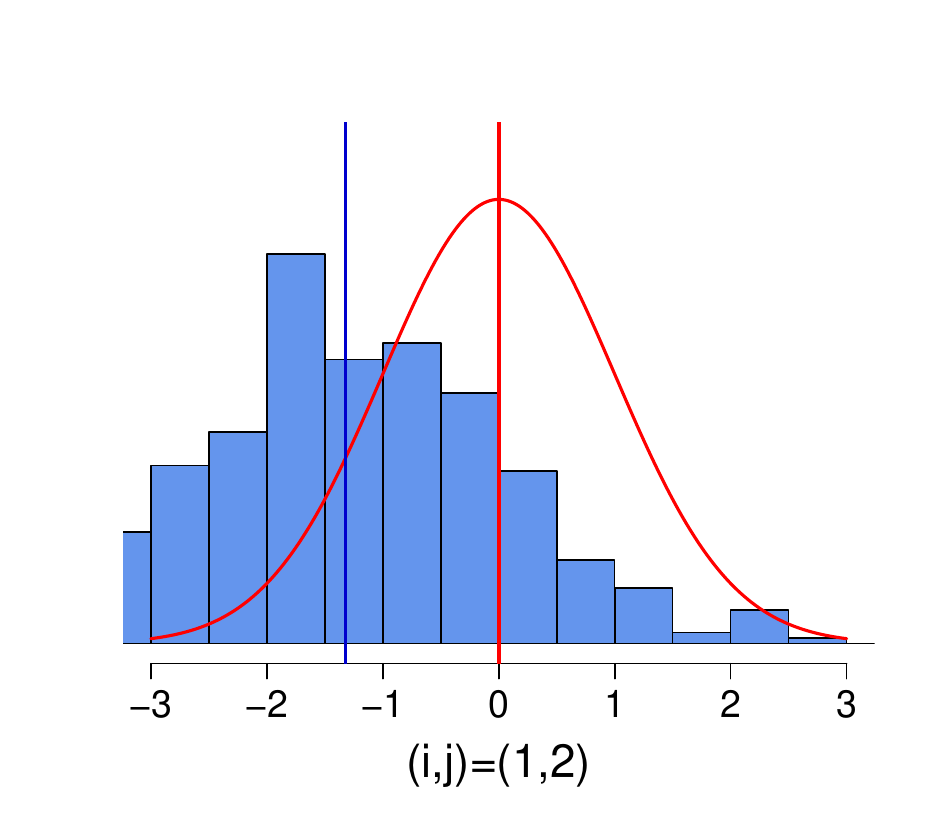}
    \end{minipage}
    \begin{minipage}{0.24\linewidth}
        \centering
        \includegraphics[width=\textwidth]{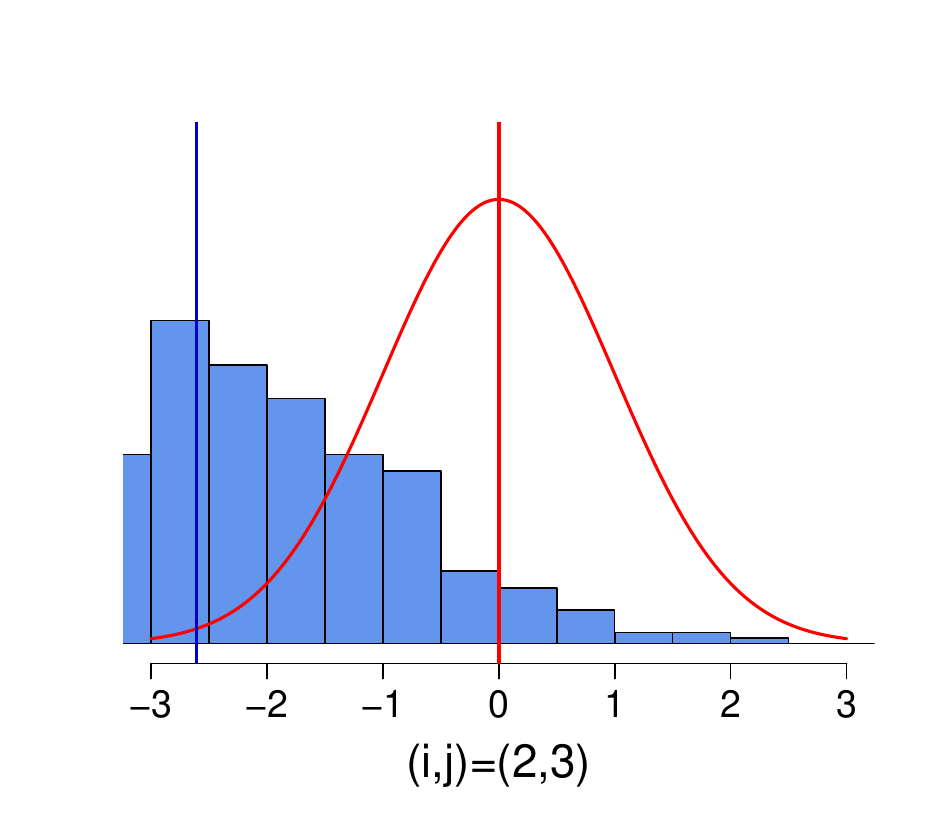}
    \end{minipage}
    \begin{minipage}{0.24\linewidth}
        \centering
        \includegraphics[width=\textwidth]{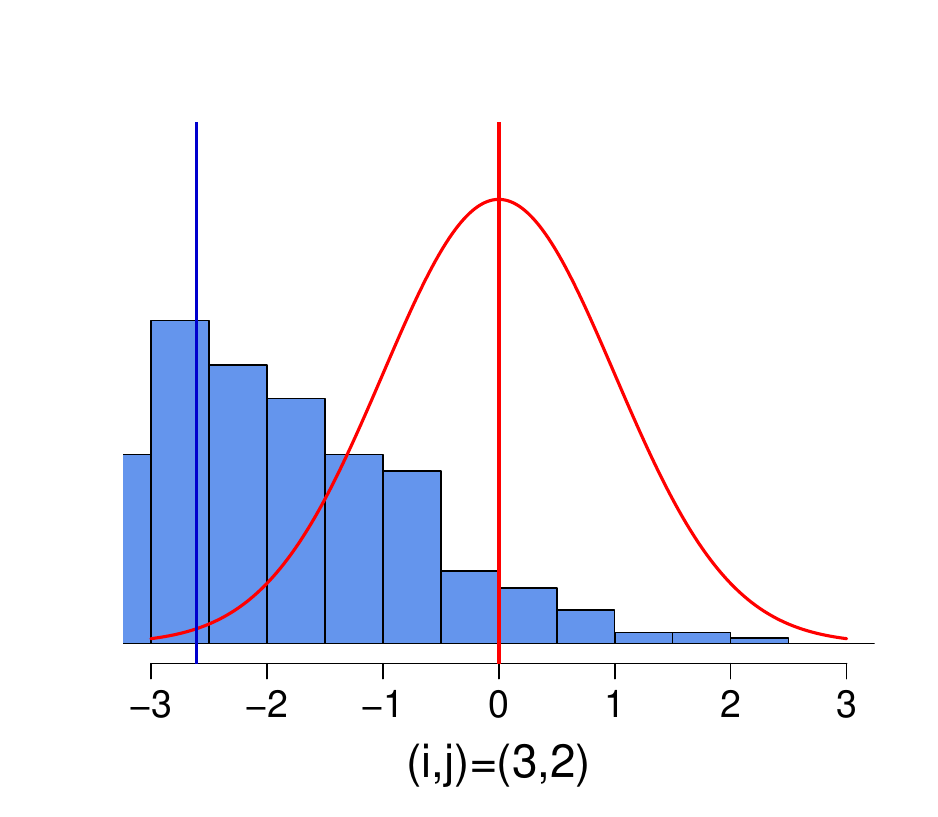}
    \end{minipage}
    \begin{minipage}{0.24\linewidth}
        \centering
        \includegraphics[width=\textwidth]{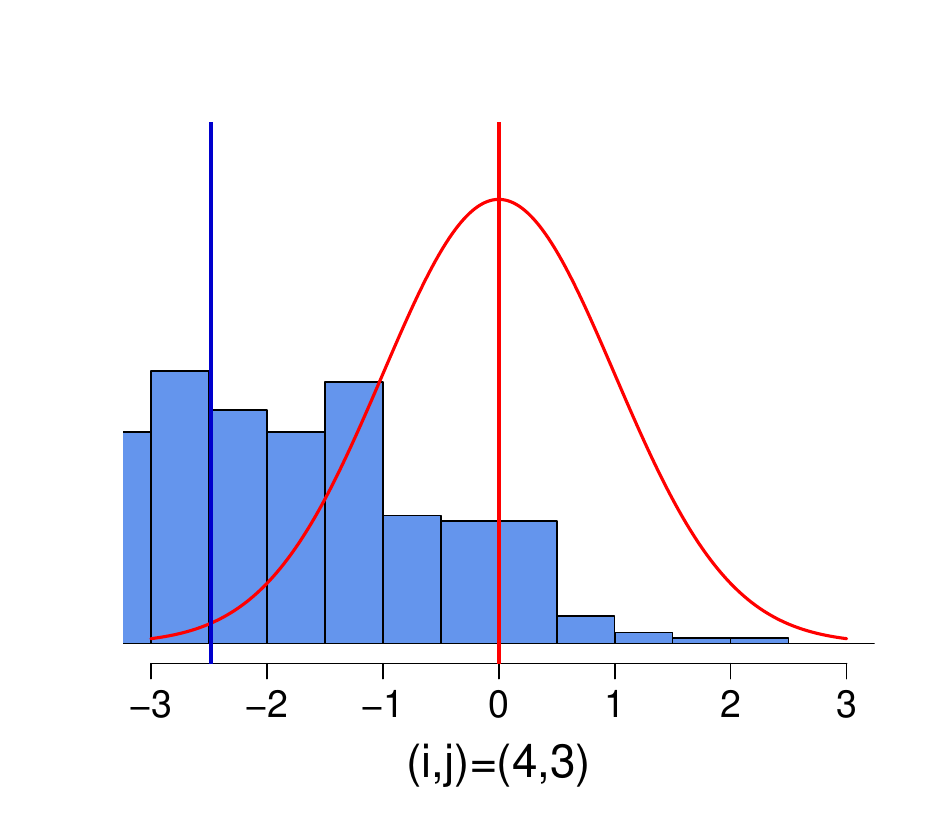}
    \end{minipage}
 \end{minipage}

  \caption*{$n=400, p=400$}
      \vspace{-0.43cm}
 \begin{minipage}{0.3\linewidth}
    \begin{minipage}{0.24\linewidth}
        \centering
        \includegraphics[width=\textwidth]{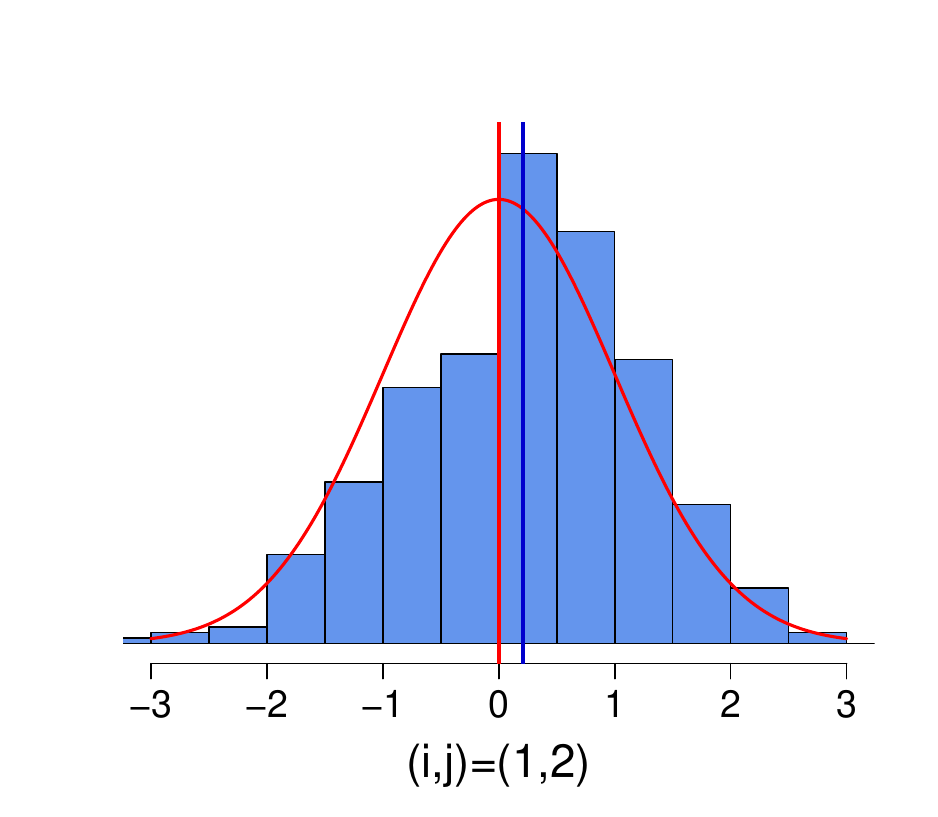}
    \end{minipage}
    \begin{minipage}{0.24\linewidth}
        \centering
        \includegraphics[width=\textwidth]{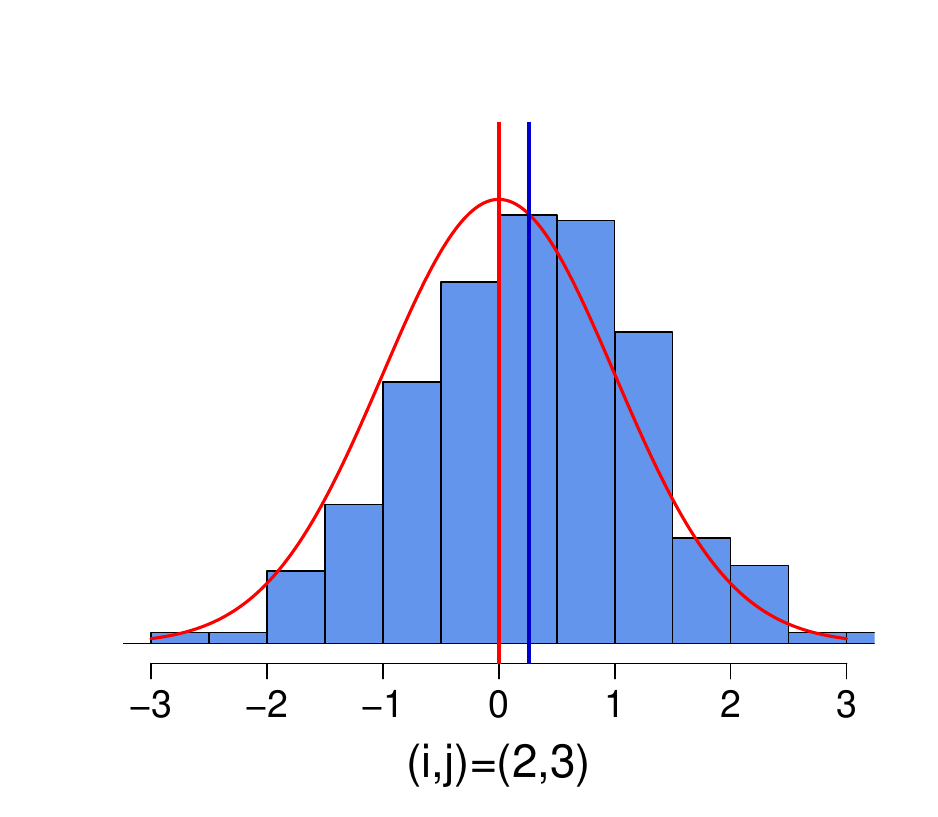}
    \end{minipage}
    \begin{minipage}{0.24\linewidth}
        \centering
        \includegraphics[width=\textwidth]{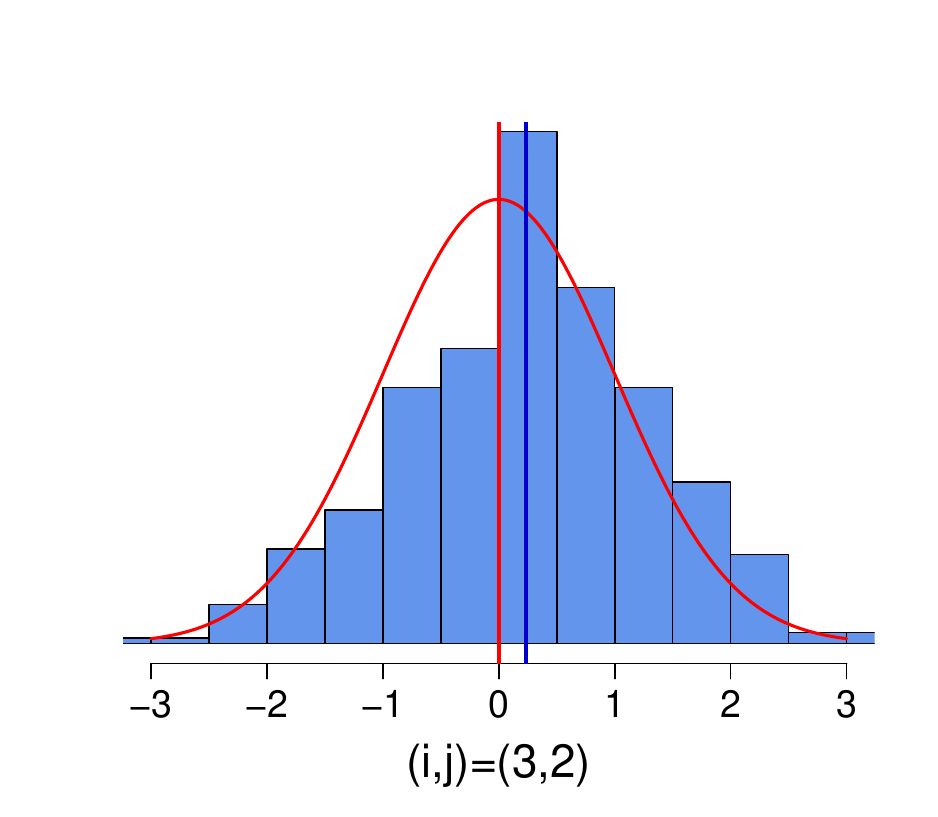}
    \end{minipage}
    \begin{minipage}{0.24\linewidth}
        \centering
        \includegraphics[width=\textwidth]{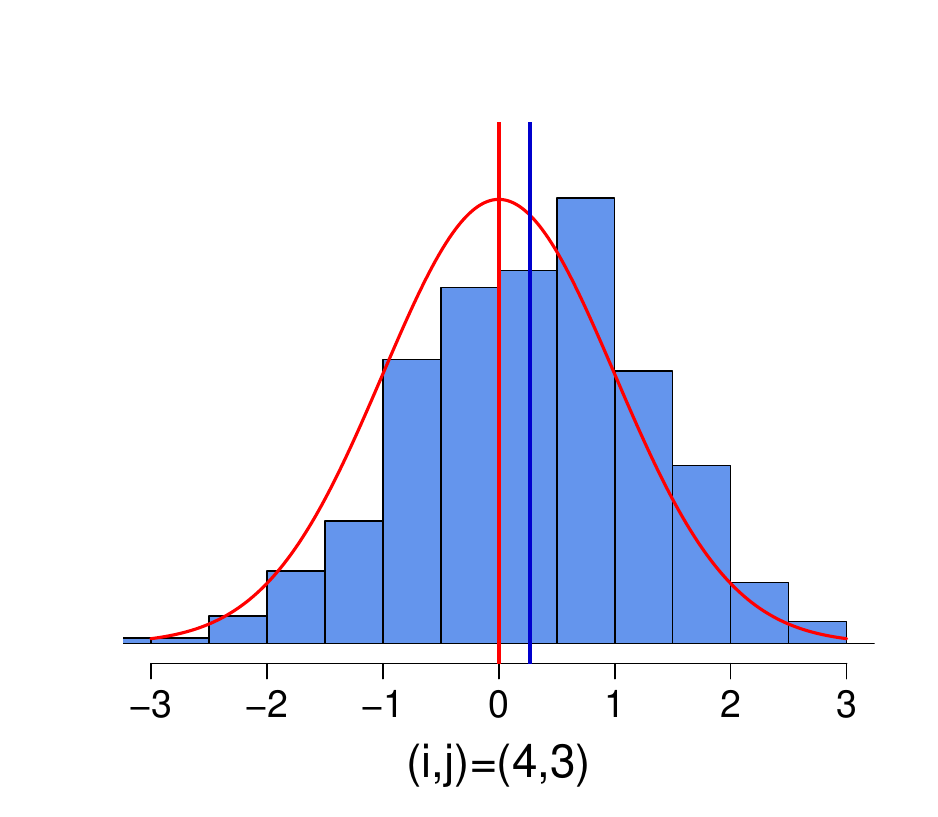}
    \end{minipage}
 \end{minipage}  
     \hspace{1cm}
 \begin{minipage}{0.3\linewidth}
    \begin{minipage}{0.24\linewidth}
        \centering
        \includegraphics[width=\textwidth]{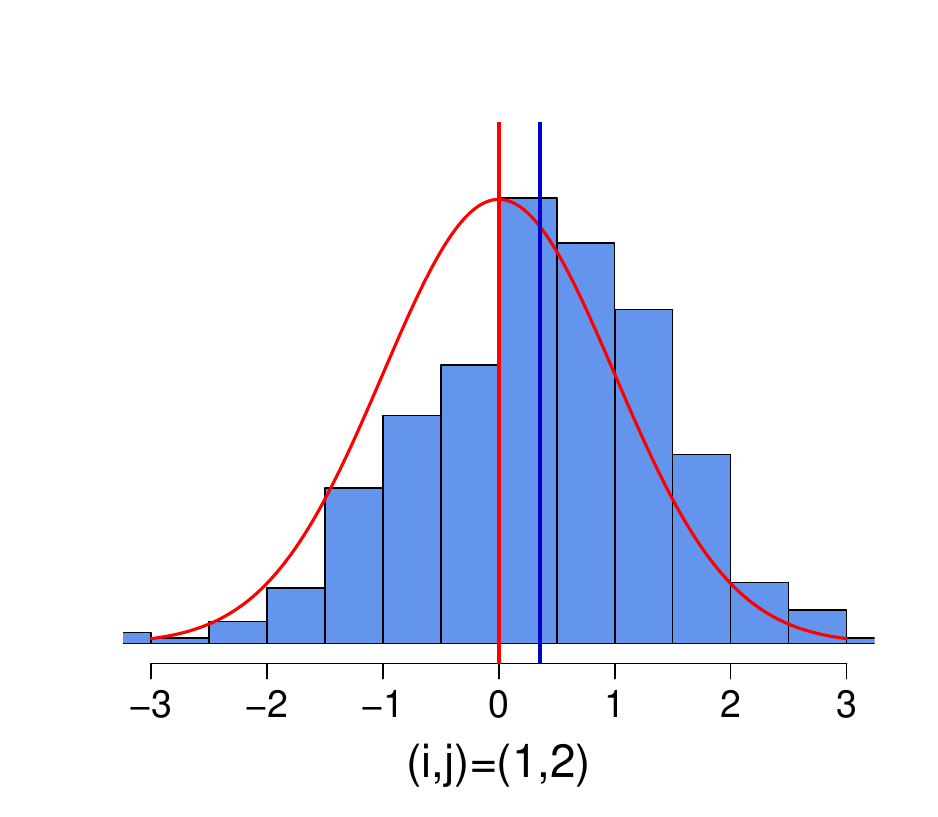}
    \end{minipage}
    \begin{minipage}{0.24\linewidth}
        \centering
        \includegraphics[width=\textwidth]{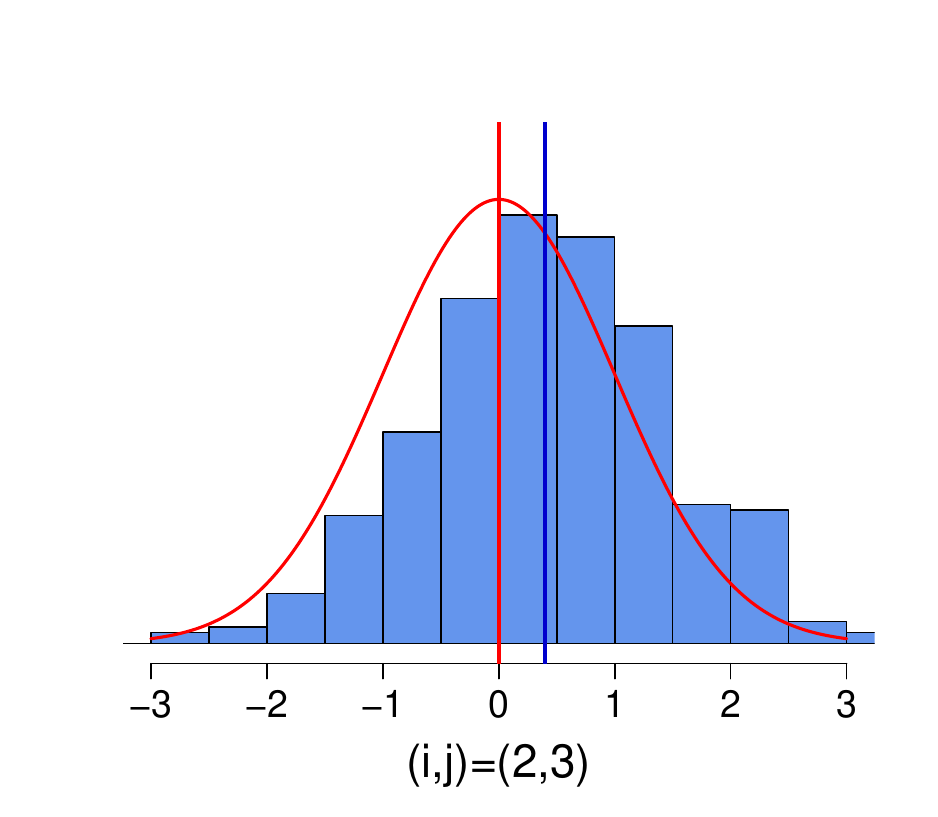}
    \end{minipage}
    \begin{minipage}{0.24\linewidth}
        \centering
        \includegraphics[width=\textwidth]{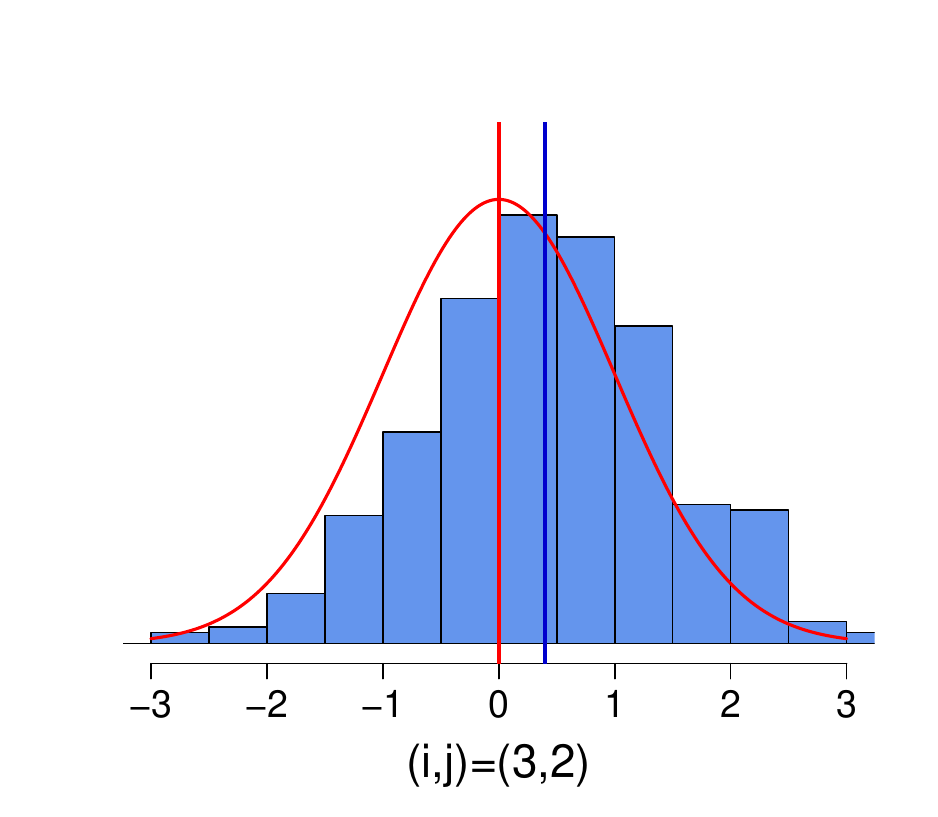}
    \end{minipage}
    \begin{minipage}{0.24\linewidth}
        \centering
        \includegraphics[width=\textwidth]{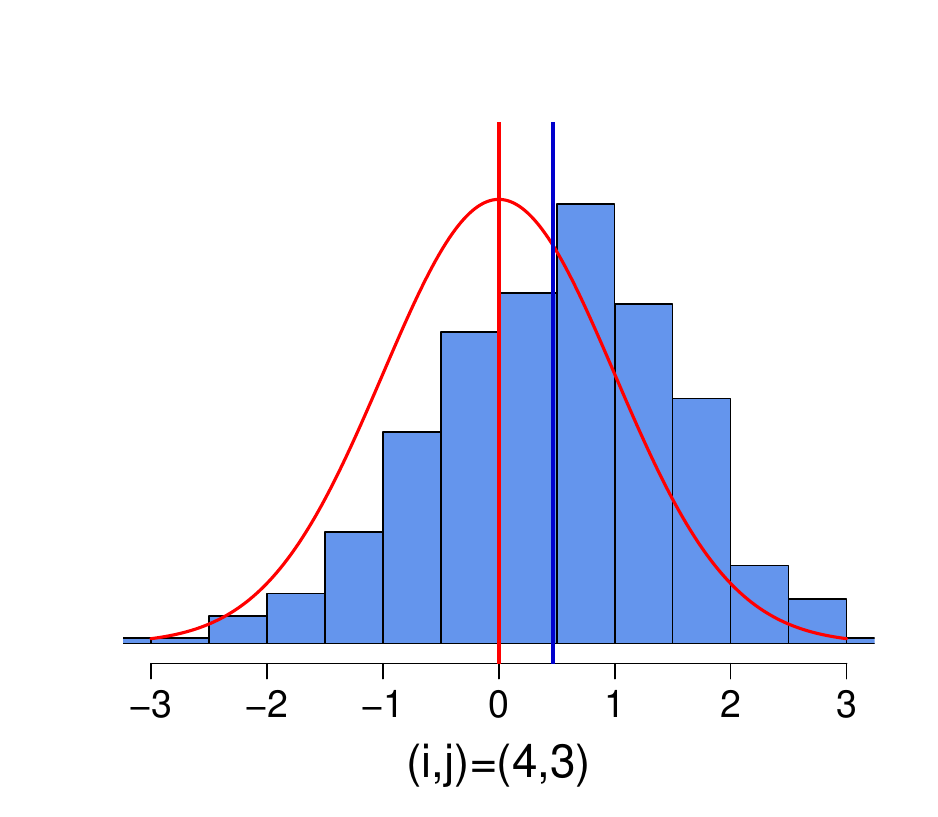}
    \end{minipage}
  \end{minipage}  
    \hspace{1cm}
 \begin{minipage}{0.3\linewidth}
    \begin{minipage}{0.24\linewidth}
        \centering
        \includegraphics[width=\textwidth]{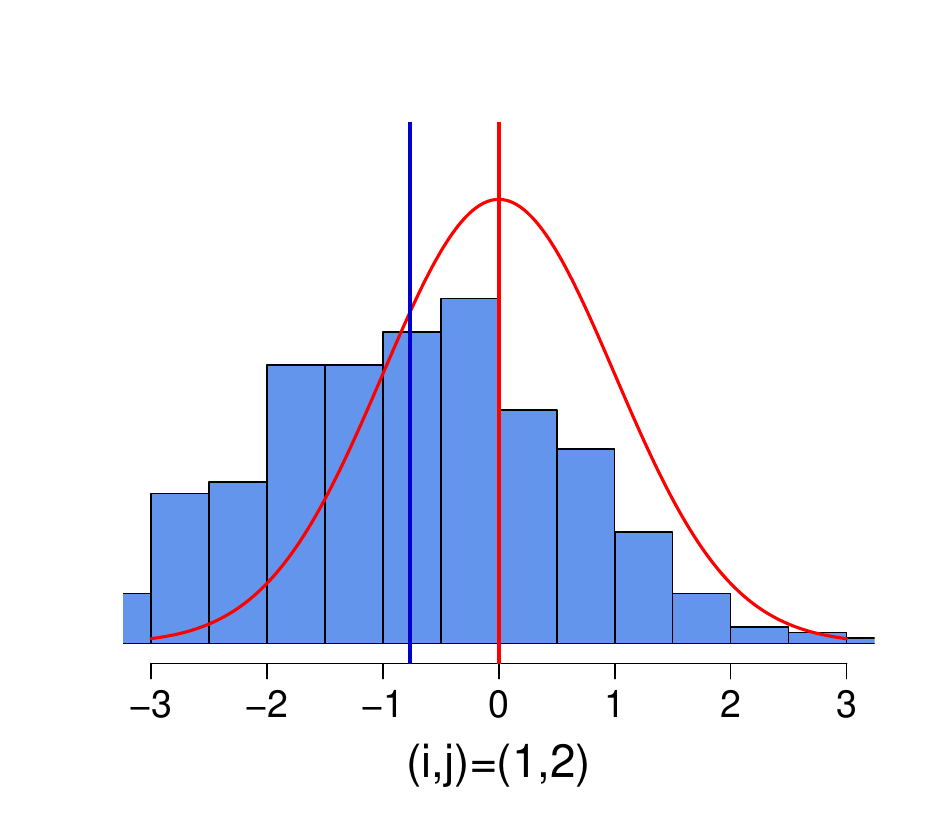}
    \end{minipage}
    \begin{minipage}{0.24\linewidth}
        \centering
        \includegraphics[width=\textwidth]{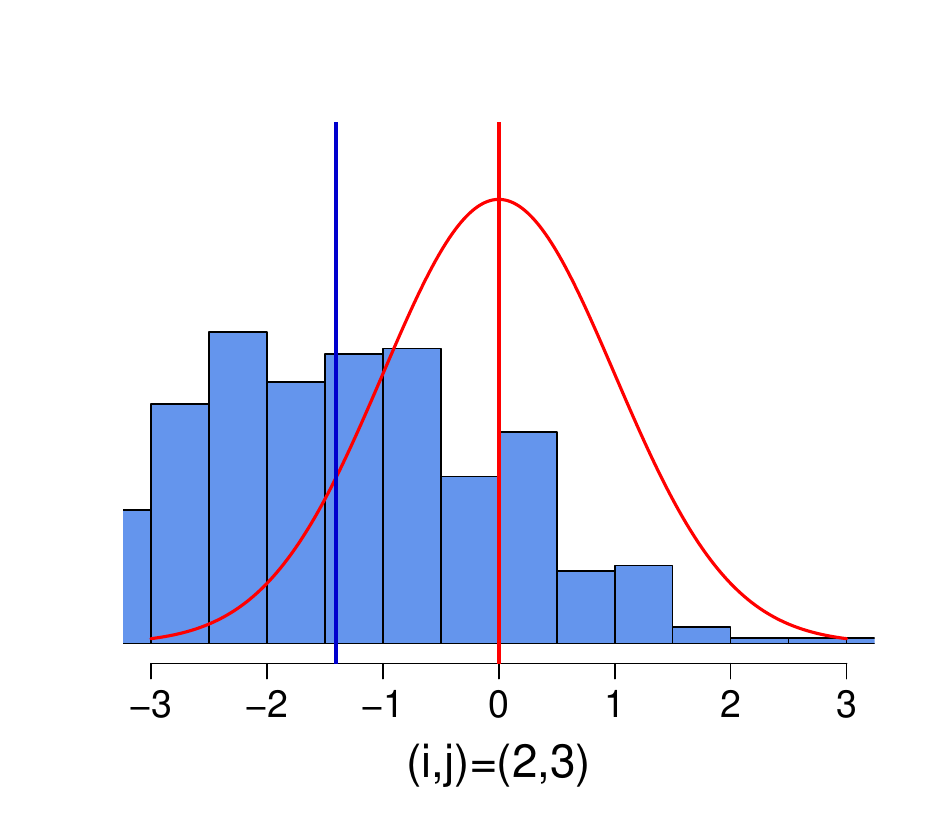}
    \end{minipage}
    \begin{minipage}{0.24\linewidth}
        \centering
        \includegraphics[width=\textwidth]{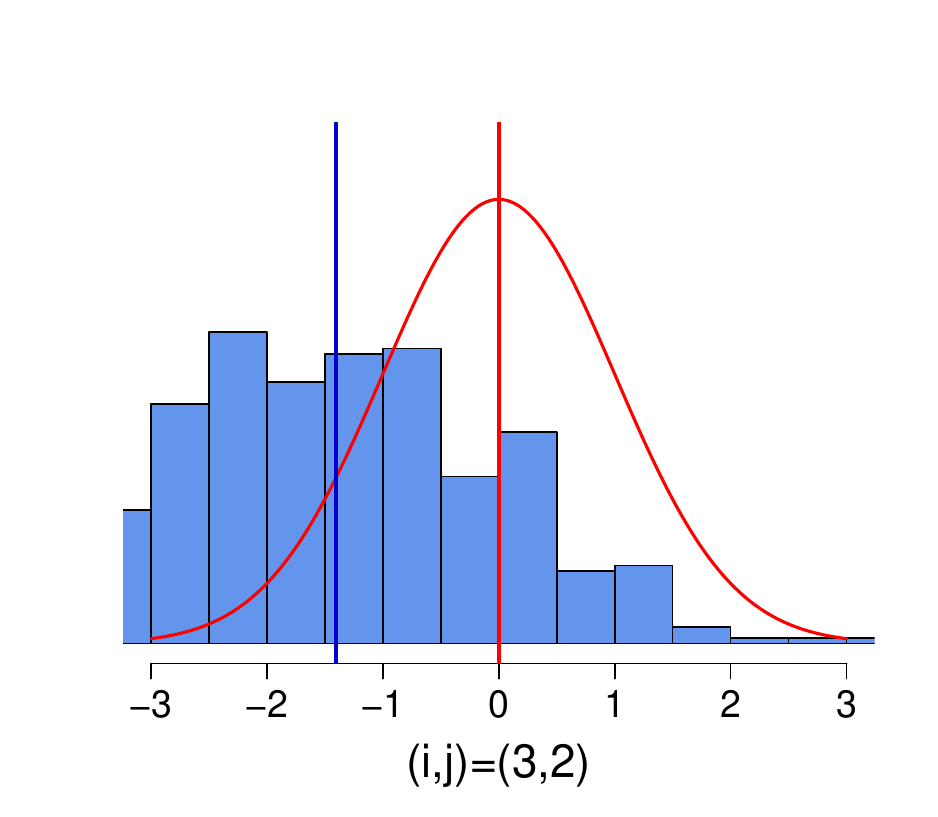}
    \end{minipage}
    \begin{minipage}{0.24\linewidth}
        \centering
        \includegraphics[width=\textwidth]{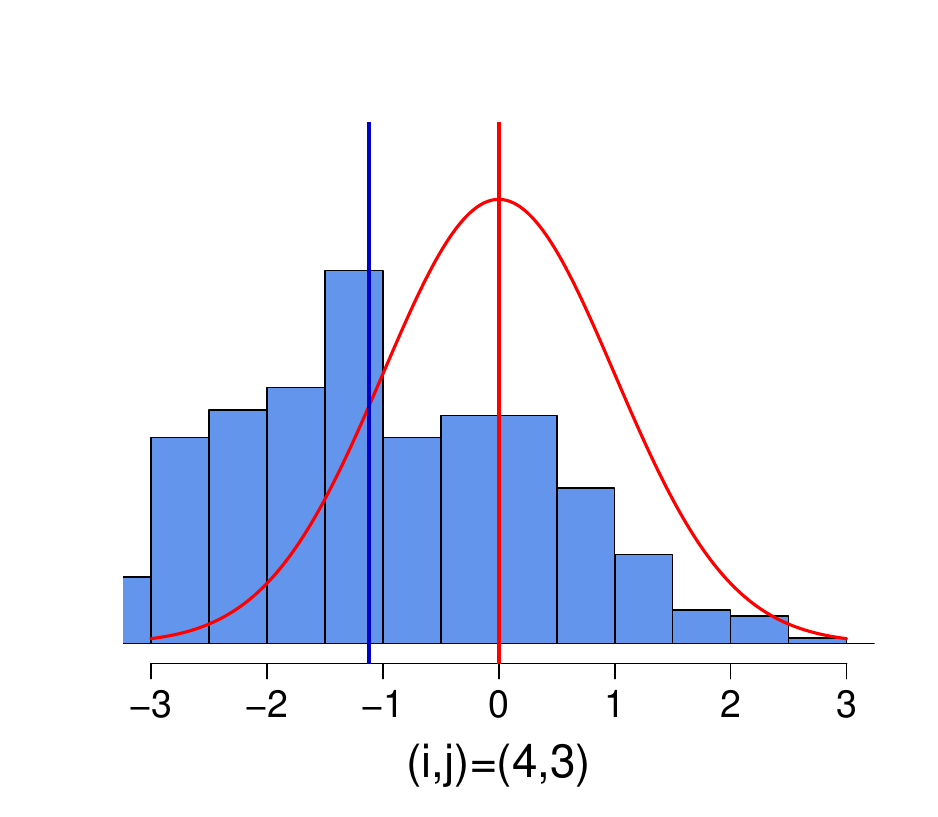}
    \end{minipage}
 \end{minipage}

  \caption*{$n=800, p=400$}
      \vspace{-0.43cm}
 \begin{minipage}{0.3\linewidth}
    \begin{minipage}{0.24\linewidth}
        \centering
        \includegraphics[width=\textwidth]{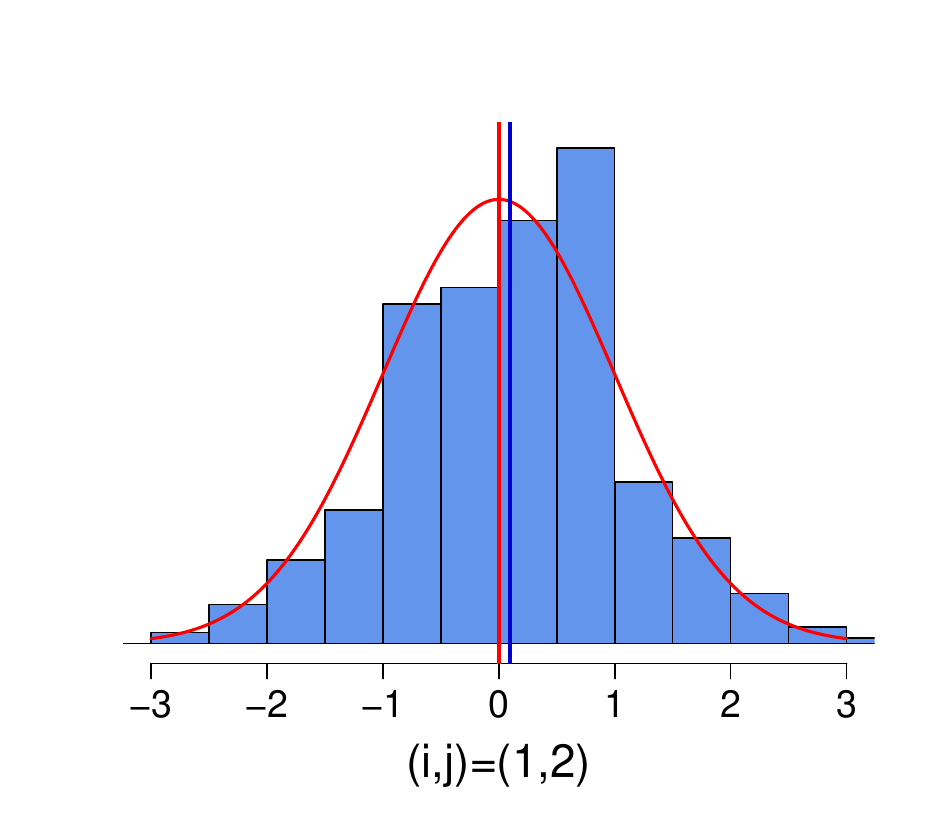}
    \end{minipage}
    \begin{minipage}{0.24\linewidth}
        \centering
        \includegraphics[width=\textwidth]{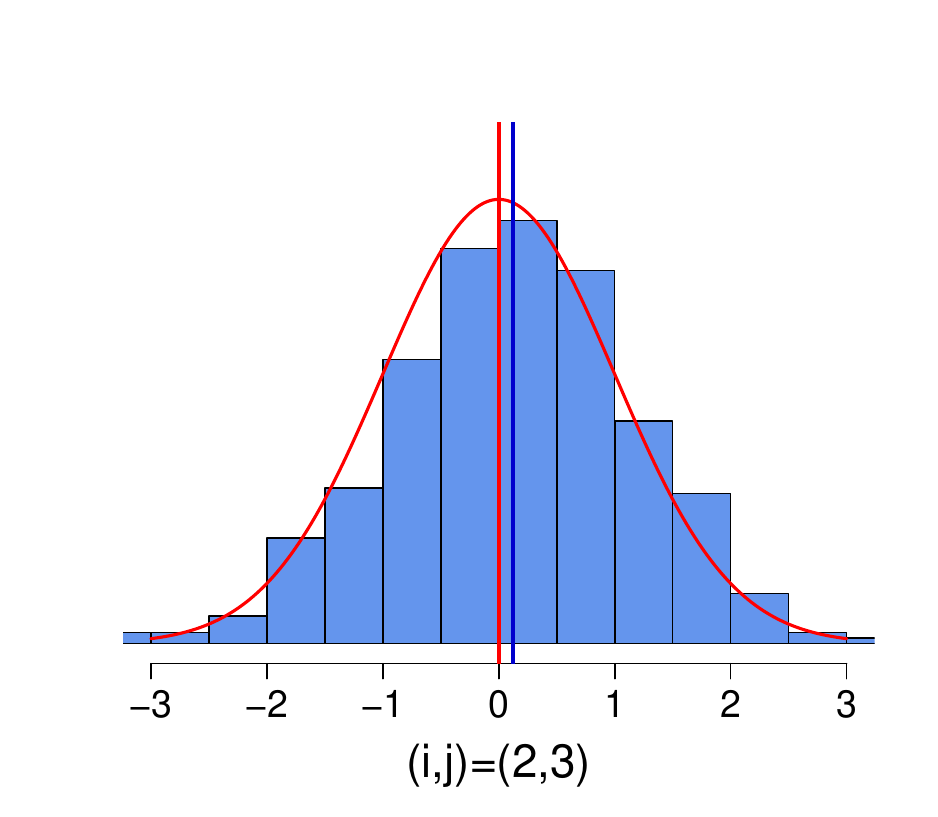}
    \end{minipage}
    \begin{minipage}{0.24\linewidth}
        \centering
        \includegraphics[width=\textwidth]{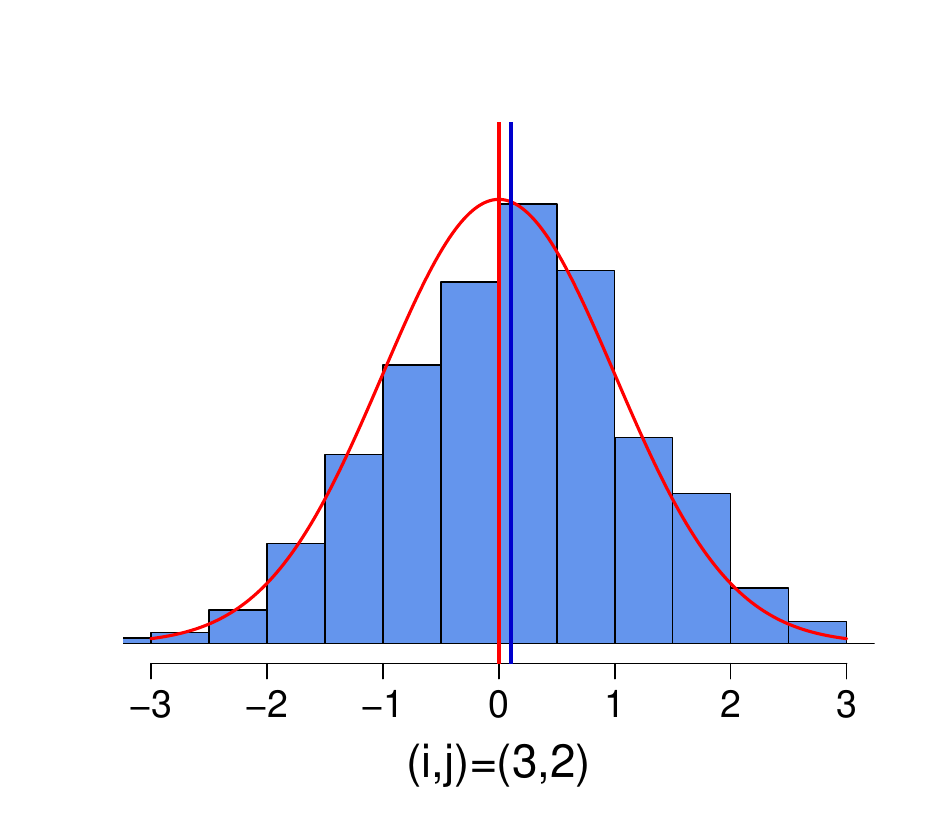}
    \end{minipage}
    \begin{minipage}{0.24\linewidth}
        \centering
        \includegraphics[width=\textwidth]{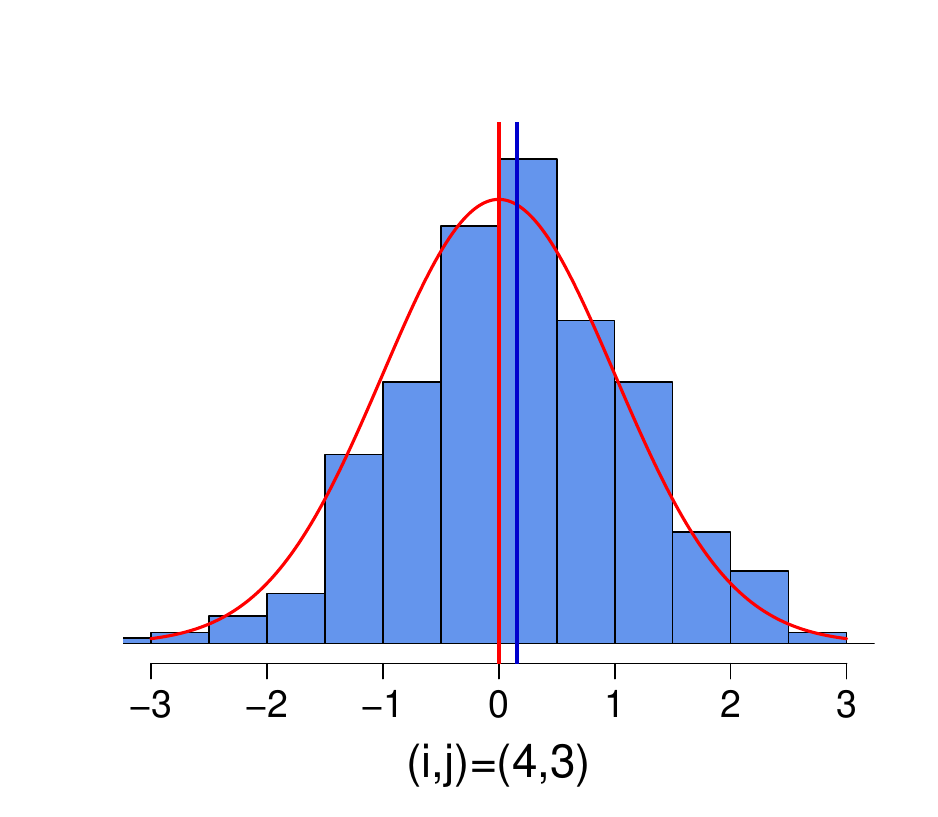}
    \end{minipage}
   \caption*{(a)~~$L_0{:}~ \widehat{\mb{\Omega}}^{\text{US}}$}
 \end{minipage} 
     \hspace{1cm}
 \begin{minipage}{0.3\linewidth}
    \begin{minipage}{0.24\linewidth}
        \centering
        \includegraphics[width=\textwidth]{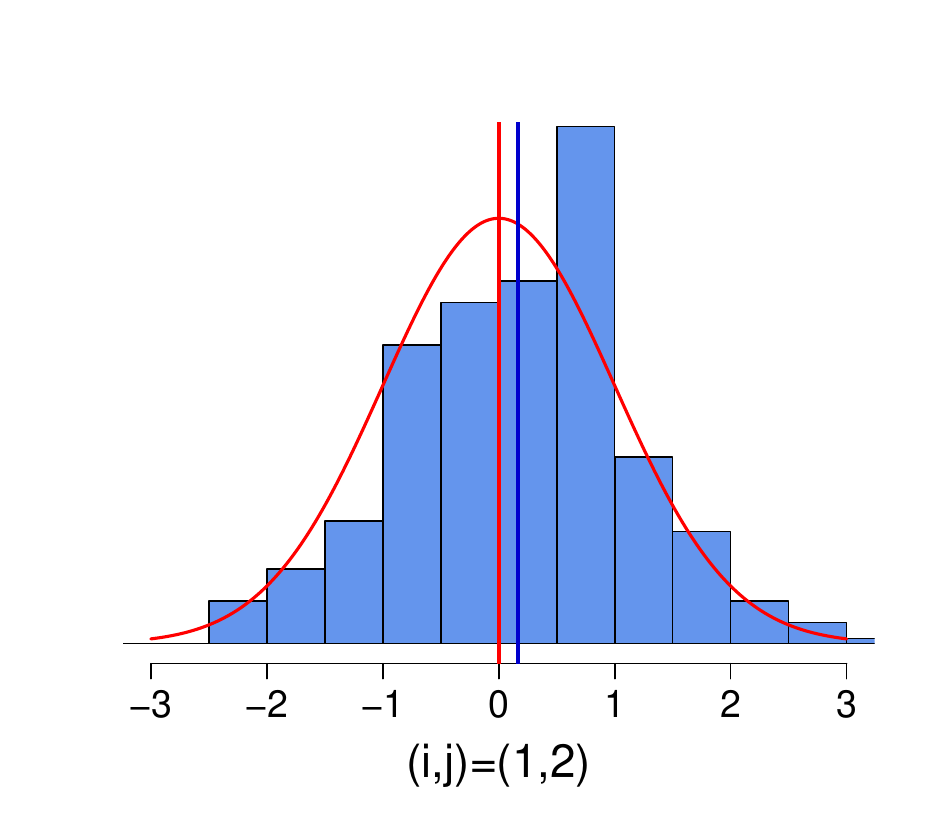}
    \end{minipage}
    \begin{minipage}{0.24\linewidth}
        \centering
        \includegraphics[width=\textwidth]{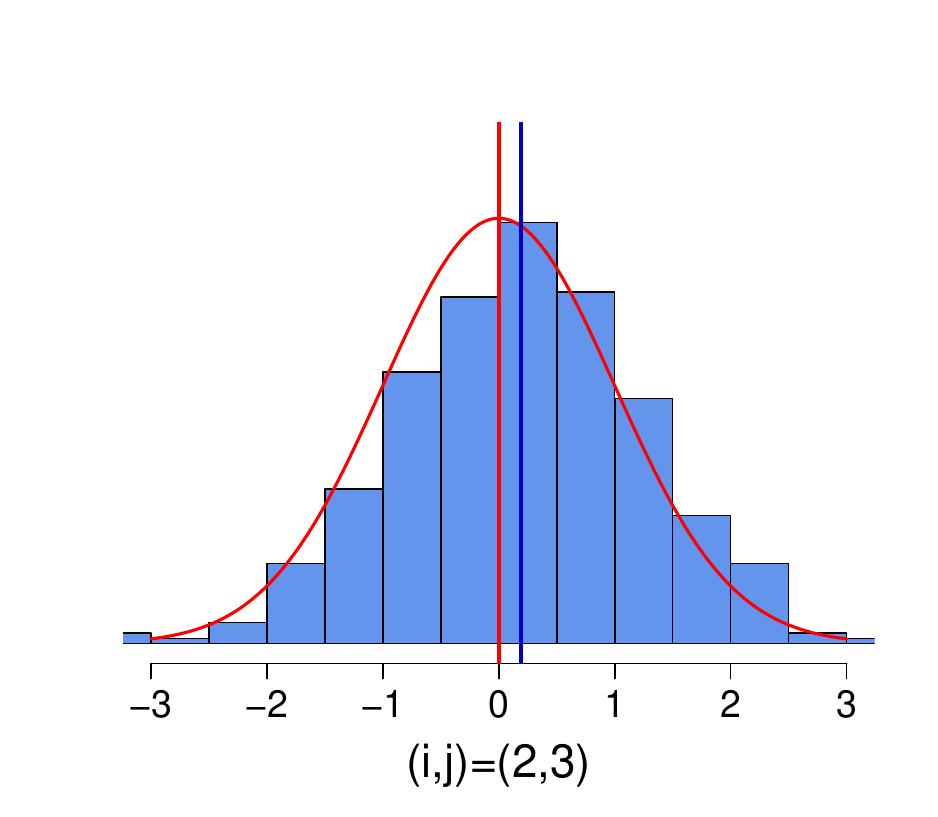}
    \end{minipage}
    \begin{minipage}{0.24\linewidth}
        \centering
        \includegraphics[width=\textwidth]{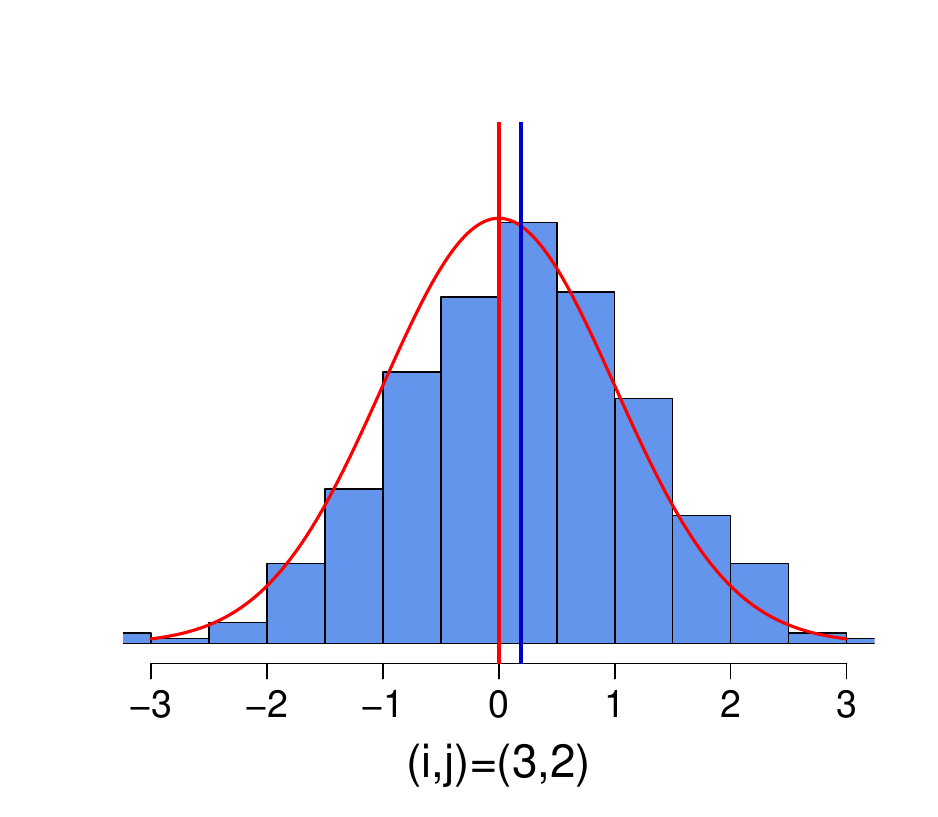}
    \end{minipage}
    \begin{minipage}{0.24\linewidth}
        \centering
        \includegraphics[width=\textwidth]{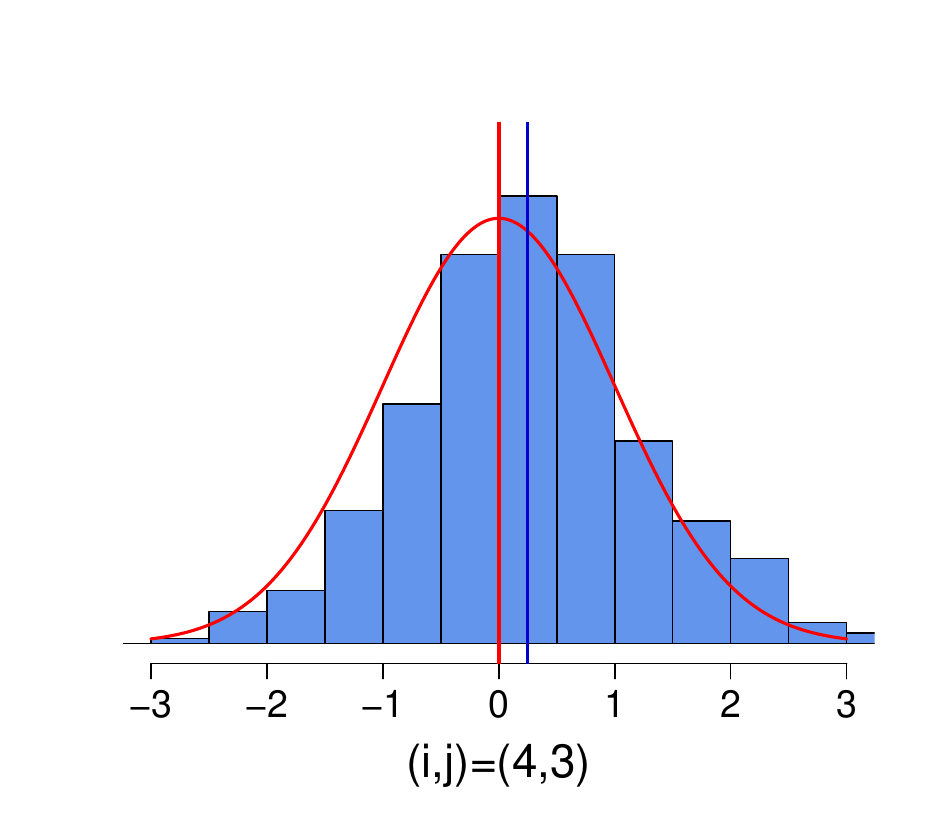}
    \end{minipage}
        \caption*{(b)~~$L_0{:}~ \widehat{\mb{T}}$}
 \end{minipage}   
      \hspace{1cm}
 \begin{minipage}{0.3\linewidth}
    \begin{minipage}{0.24\linewidth}
        \centering
        \includegraphics[width=\textwidth]{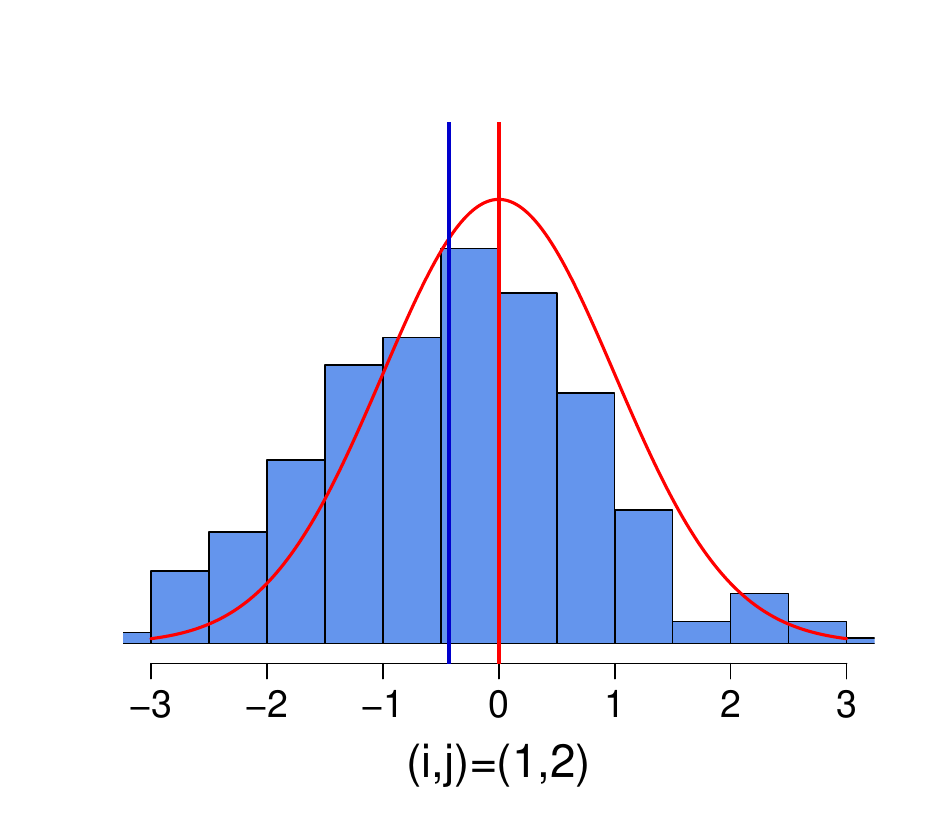}
    \end{minipage}
    \begin{minipage}{0.24\linewidth}
        \centering
        \includegraphics[width=\textwidth]{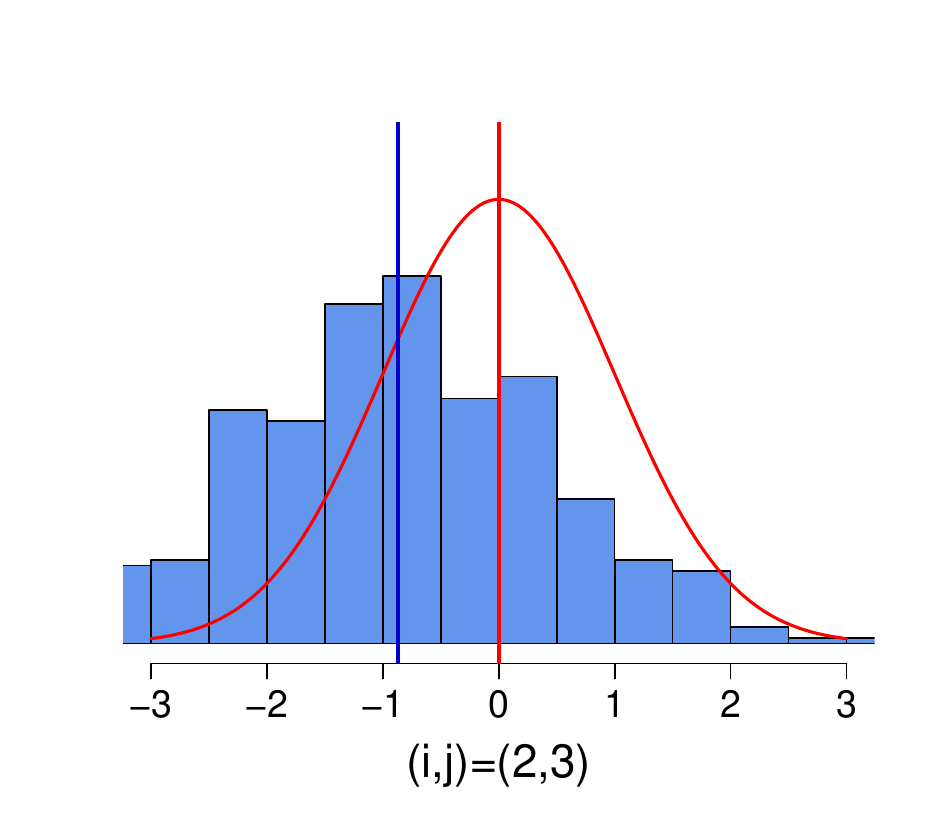}
    \end{minipage}
    \begin{minipage}{0.24\linewidth}
        \centering
        \includegraphics[width=\textwidth]{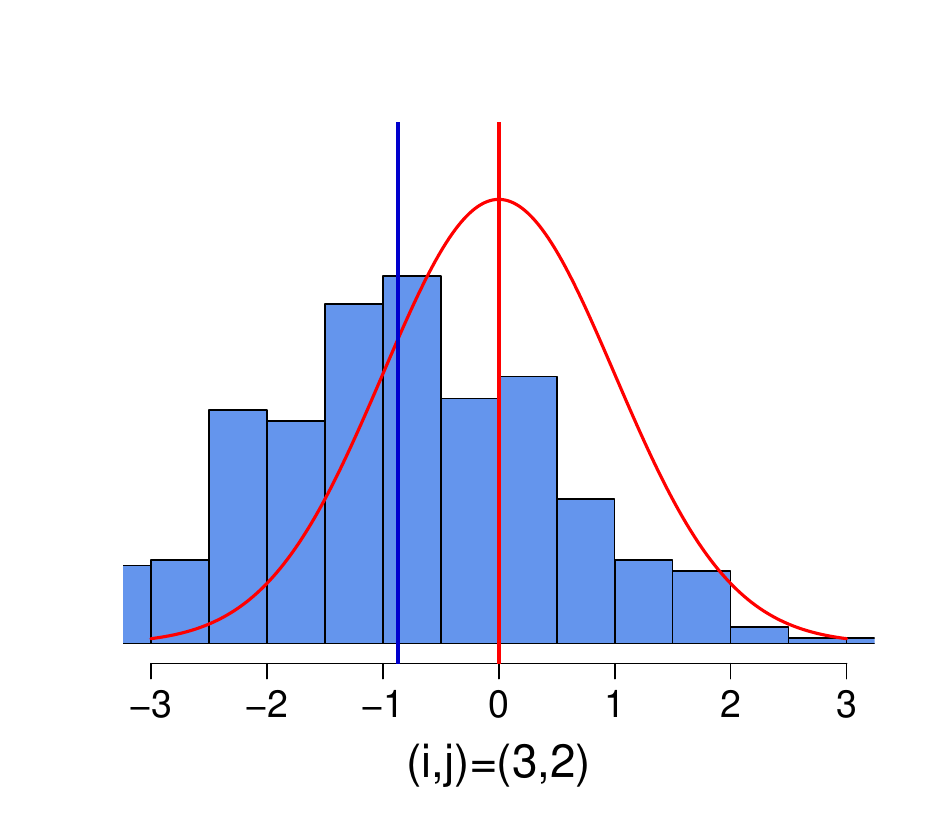}
    \end{minipage}
    \begin{minipage}{0.24\linewidth}
        \centering
        \includegraphics[width=\textwidth]{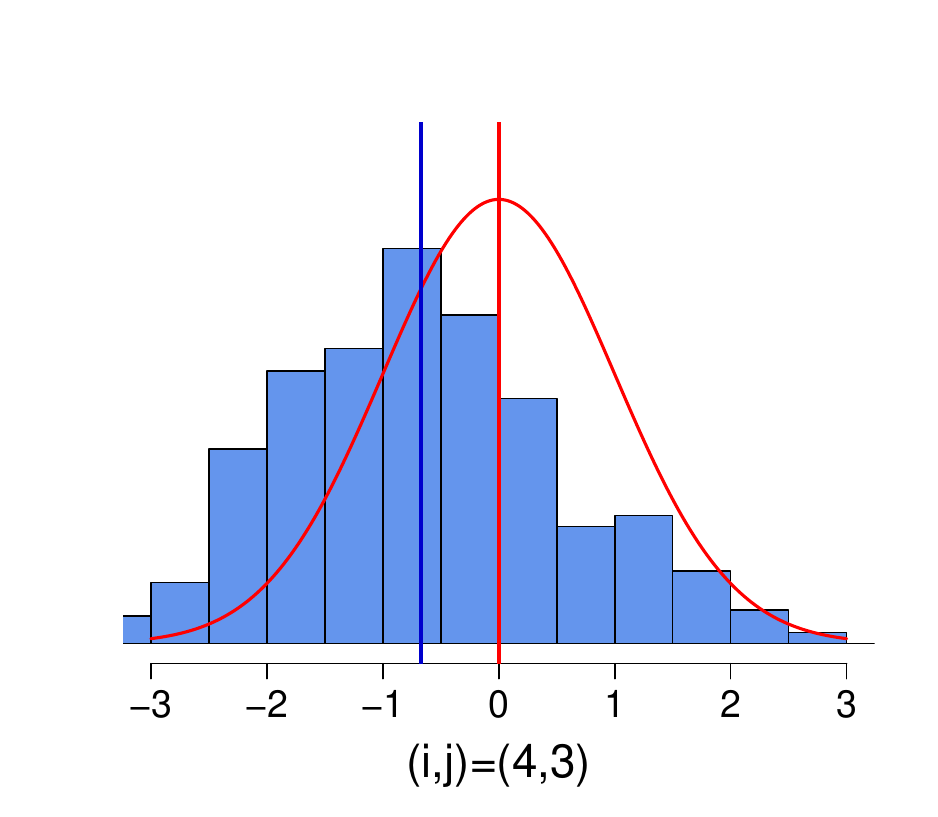}
    \end{minipage}
    \caption*{(c)~~$L_1{:}~ \widehat{\mb{T}}$}
     \end{minipage}   
     \caption{Histograms of $\big(\sqrt{n}(\widehat{\mb{\Omega}}_{ij}^{(m)}-\mb{\Omega}_{ij})/\widehat{\sigma}_{\mb{\Omega}_{ij}}^{(m)}\big)_{m=1}^{400}$ under sub-Gaussian band graph settings.}
\label{fig: normalplot sub-Gaussian band}
\end{sidewaysfigure}

%sub-Gaussian random
 \begin{sidewaysfigure}[th!]
  \caption*{$n=200, p=200$}
      \vspace{-0.43cm}
 \begin{minipage}{0.3\linewidth}
    \begin{minipage}{0.24\linewidth}
        \centering
        \includegraphics[width=\textwidth]{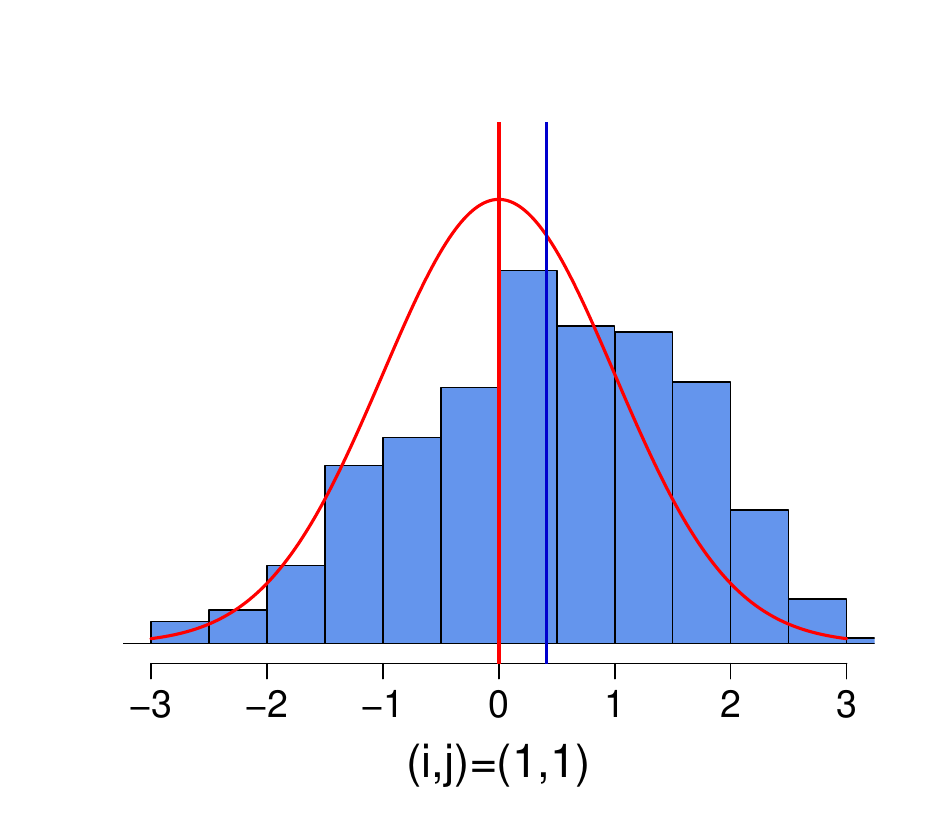}
    \end{minipage}
    \begin{minipage}{0.24\linewidth}
        \centering
        \includegraphics[width=\textwidth]{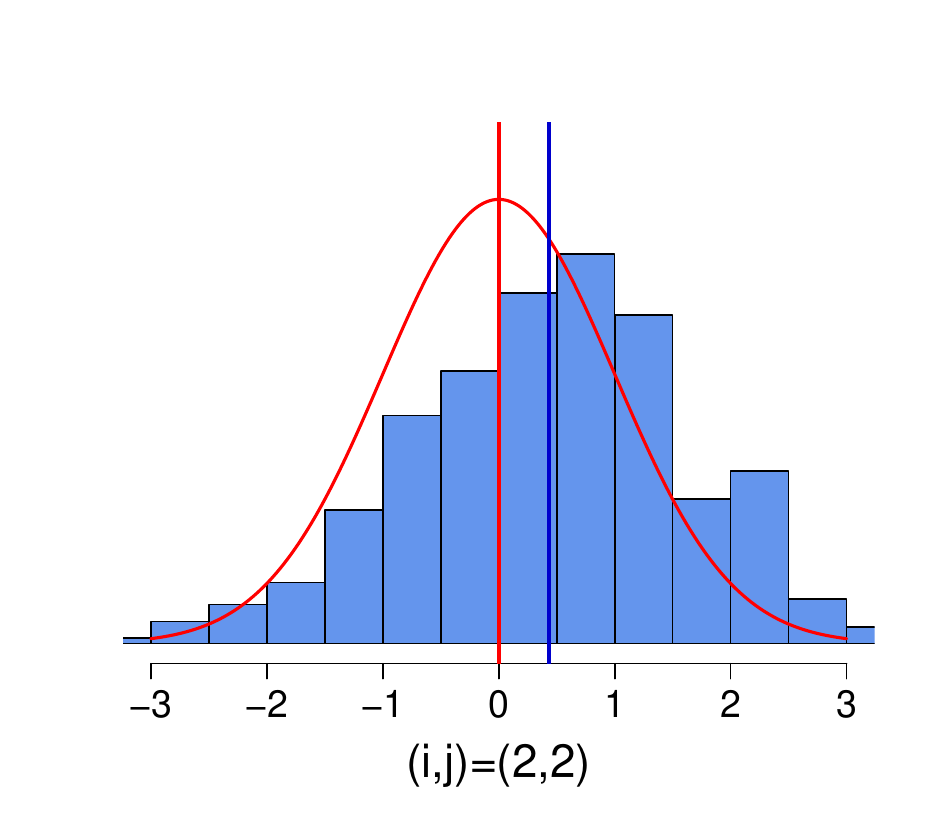}
    \end{minipage}
    \begin{minipage}{0.24\linewidth}
        \centering
        \includegraphics[width=\textwidth]{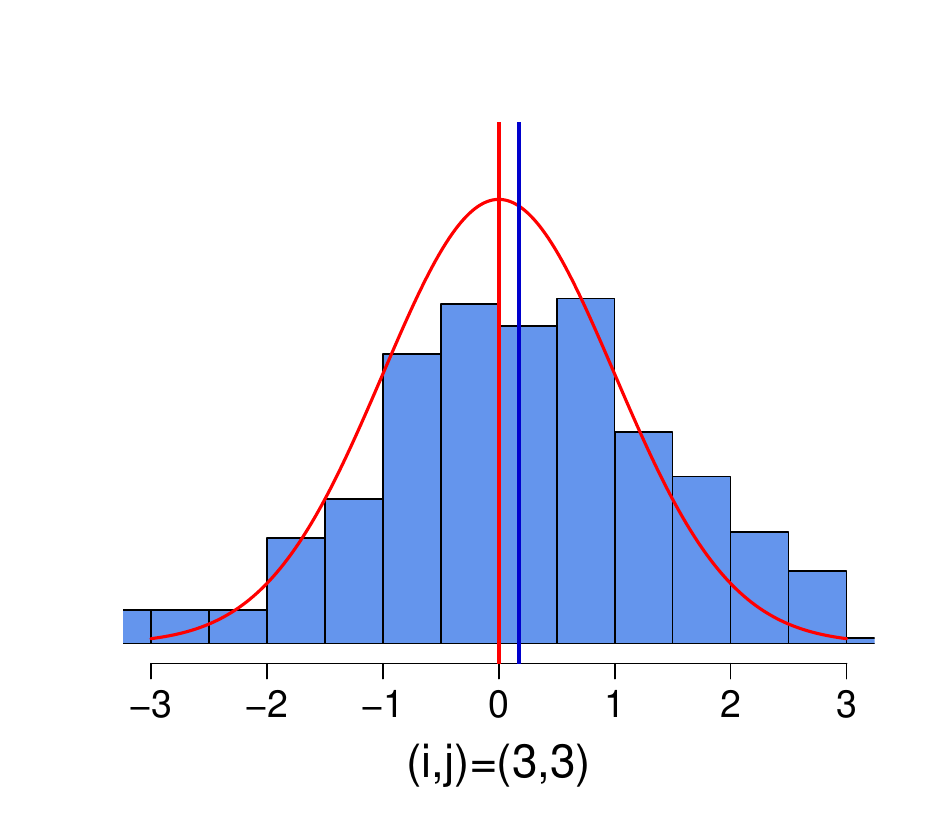}
    \end{minipage}
    \begin{minipage}{0.24\linewidth}
        \centering
        \includegraphics[width=\textwidth]{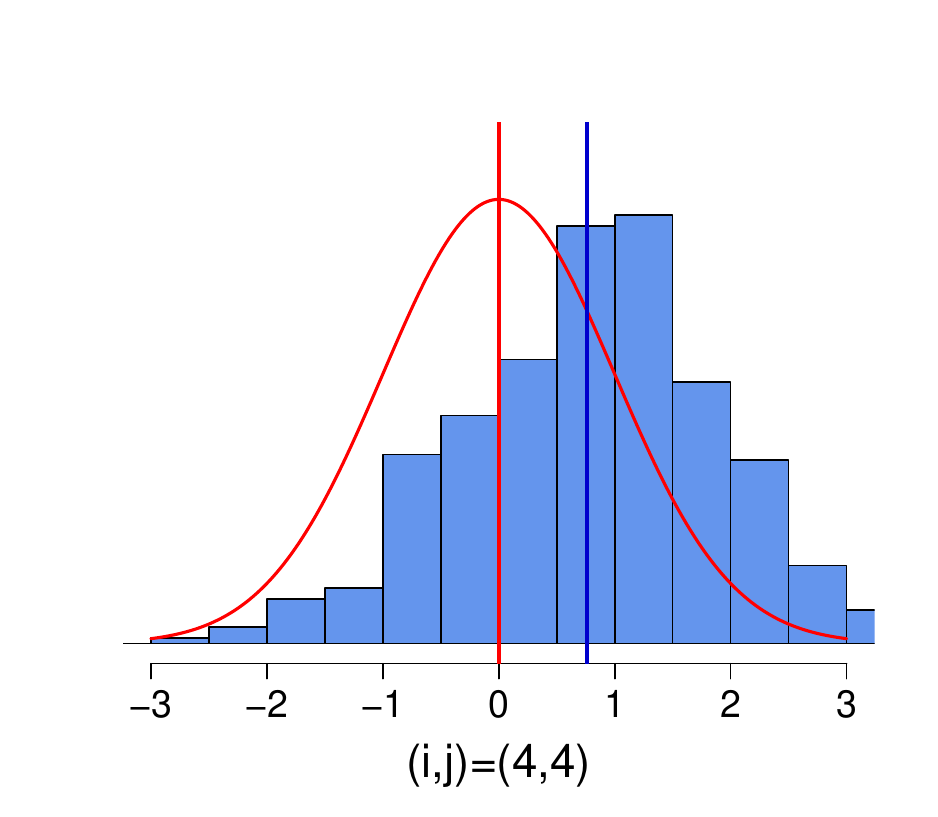}
    \end{minipage}
 \end{minipage}
 \hspace{1cm}
 \begin{minipage}{0.3\linewidth}
    \begin{minipage}{0.24\linewidth}
        \centering
        \includegraphics[width=\textwidth]{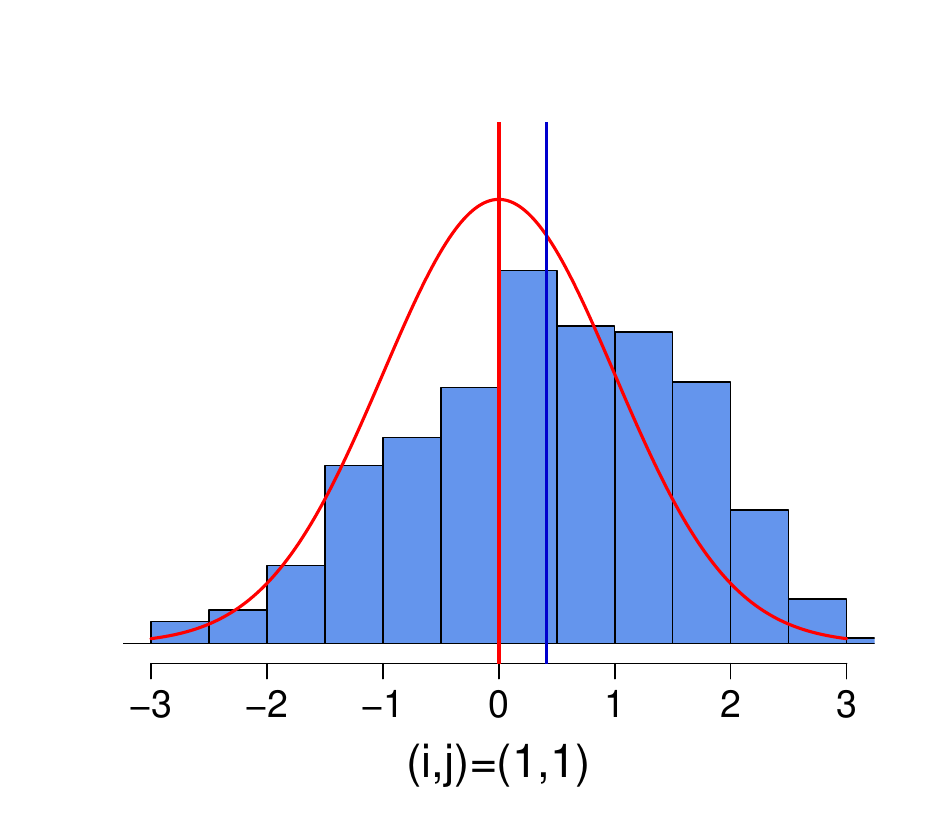}
    \end{minipage}
    \begin{minipage}{0.24\linewidth}
        \centering
        \includegraphics[width=\textwidth]{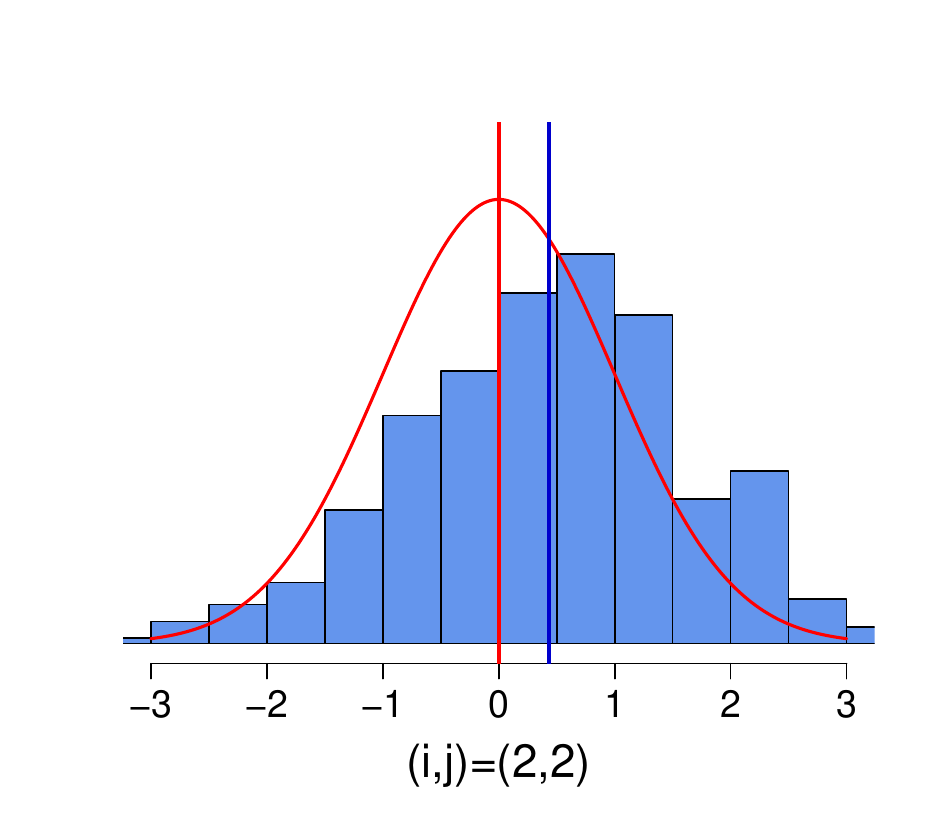}
    \end{minipage}
    \begin{minipage}{0.24\linewidth}
        \centering
        \includegraphics[width=\textwidth]{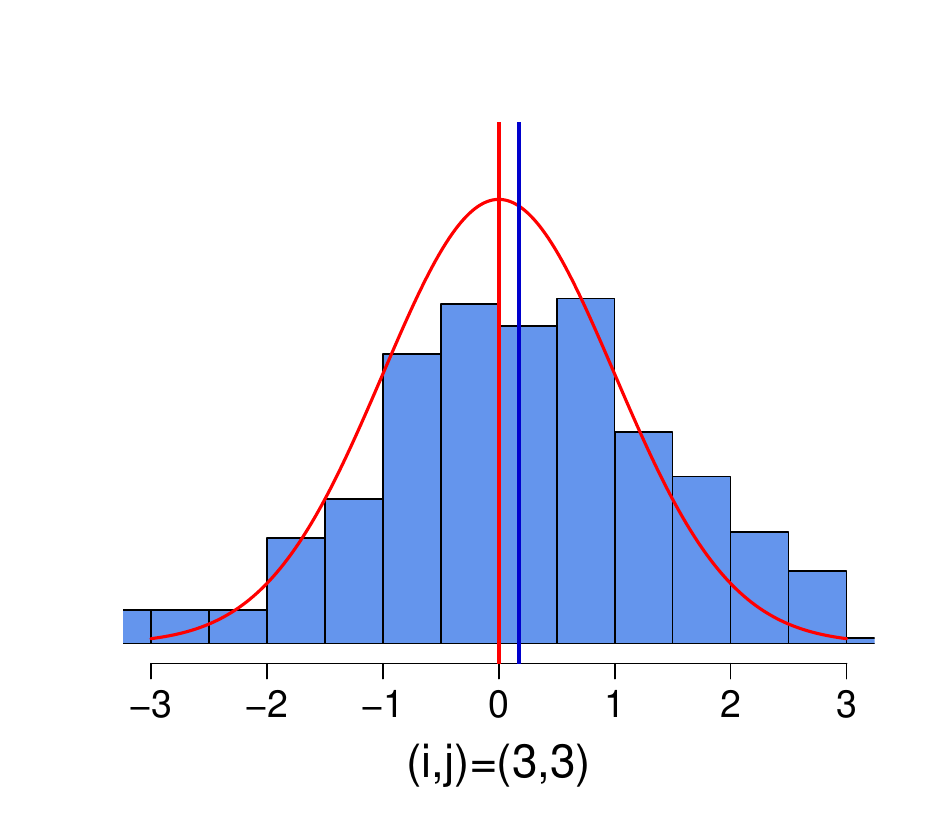}
    \end{minipage}
    \begin{minipage}{0.24\linewidth}
        \centering
        \includegraphics[width=\textwidth]{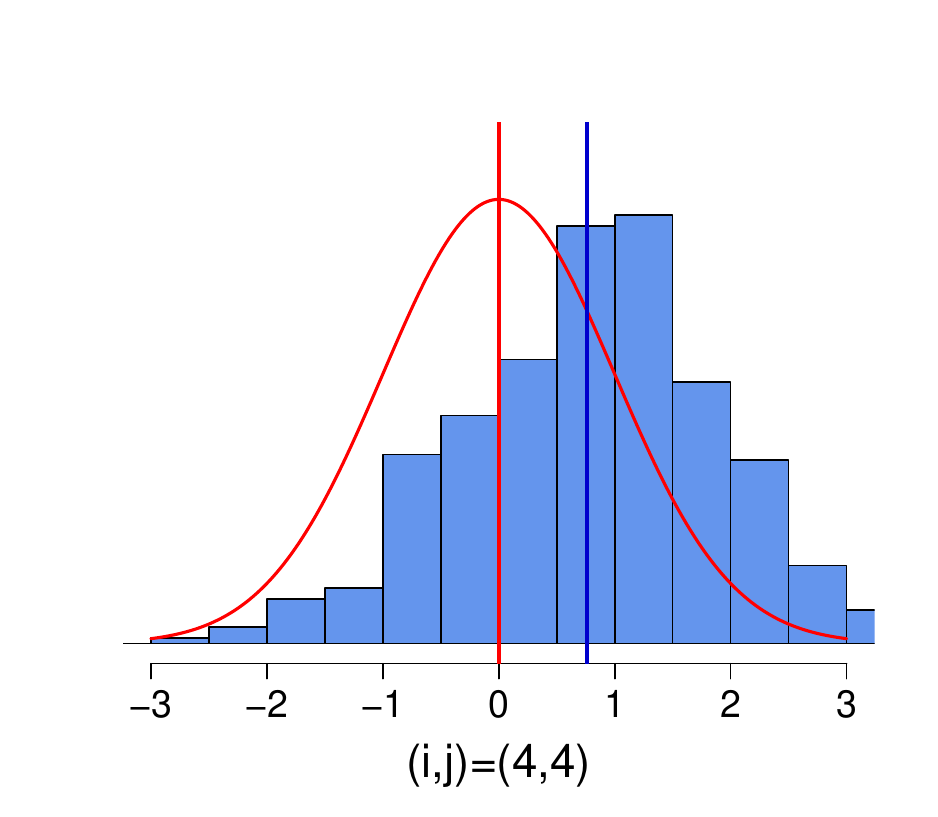}
    \end{minipage}    
 \end{minipage}
  \hspace{1cm}
 \begin{minipage}{0.3\linewidth}
     \begin{minipage}{0.24\linewidth}
        \centering
        \includegraphics[width=\textwidth]{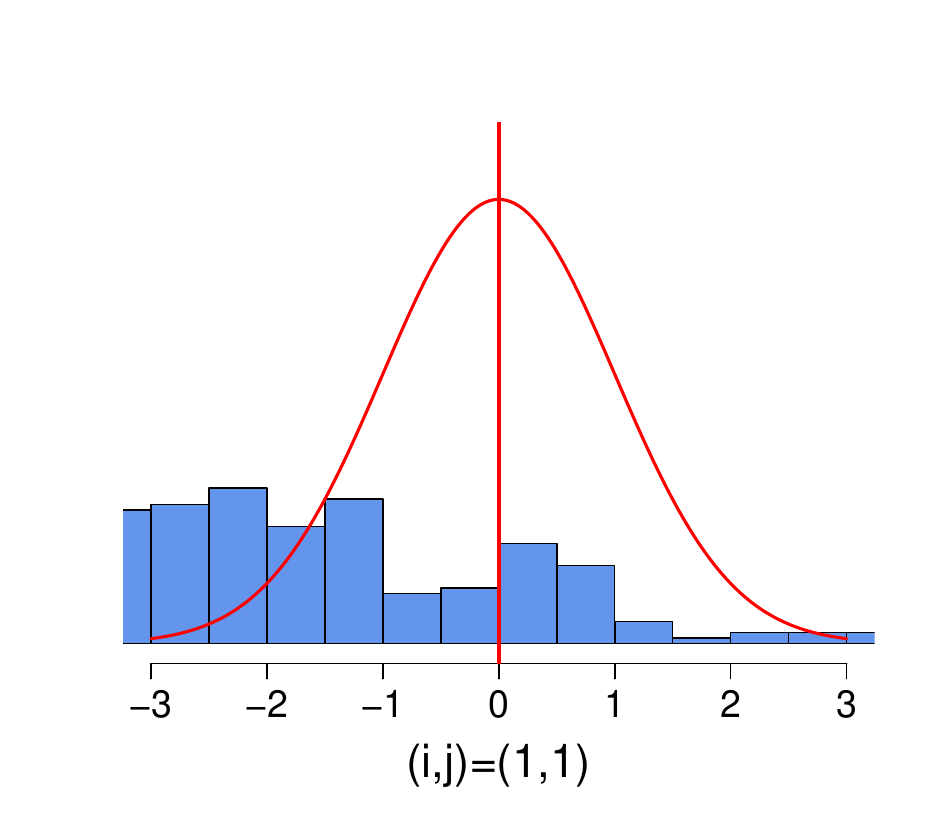}
    \end{minipage}
    \begin{minipage}{0.24\linewidth}
        \centering
        \includegraphics[width=\textwidth]{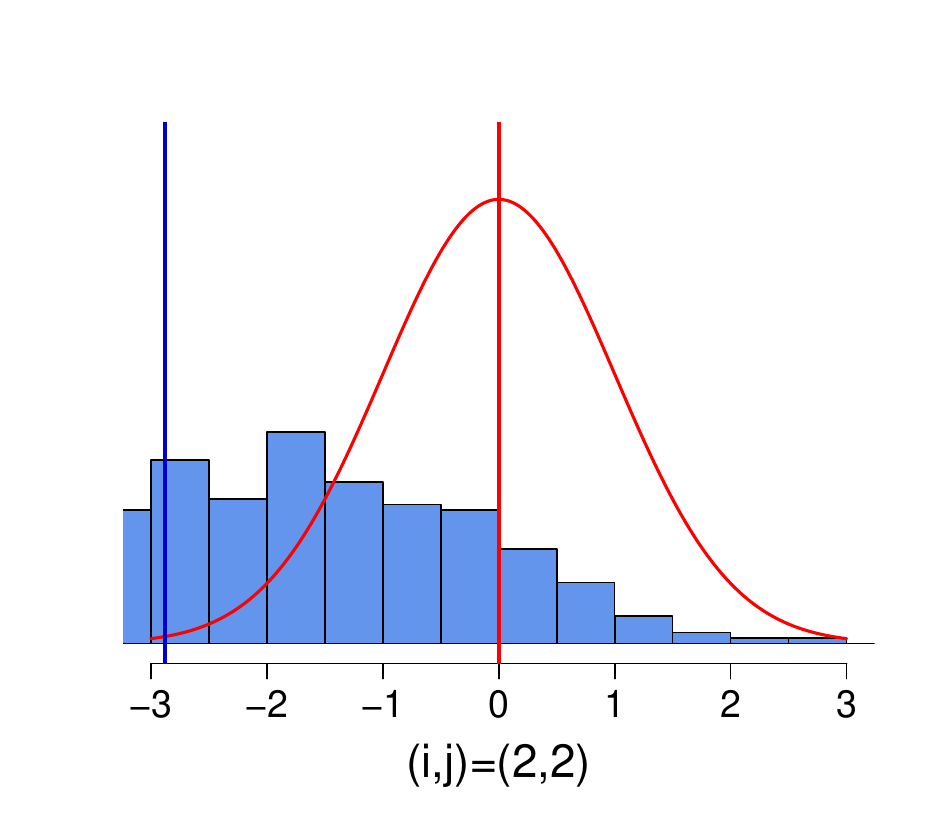}
    \end{minipage}
    \begin{minipage}{0.24\linewidth}
        \centering
        \includegraphics[width=\textwidth]{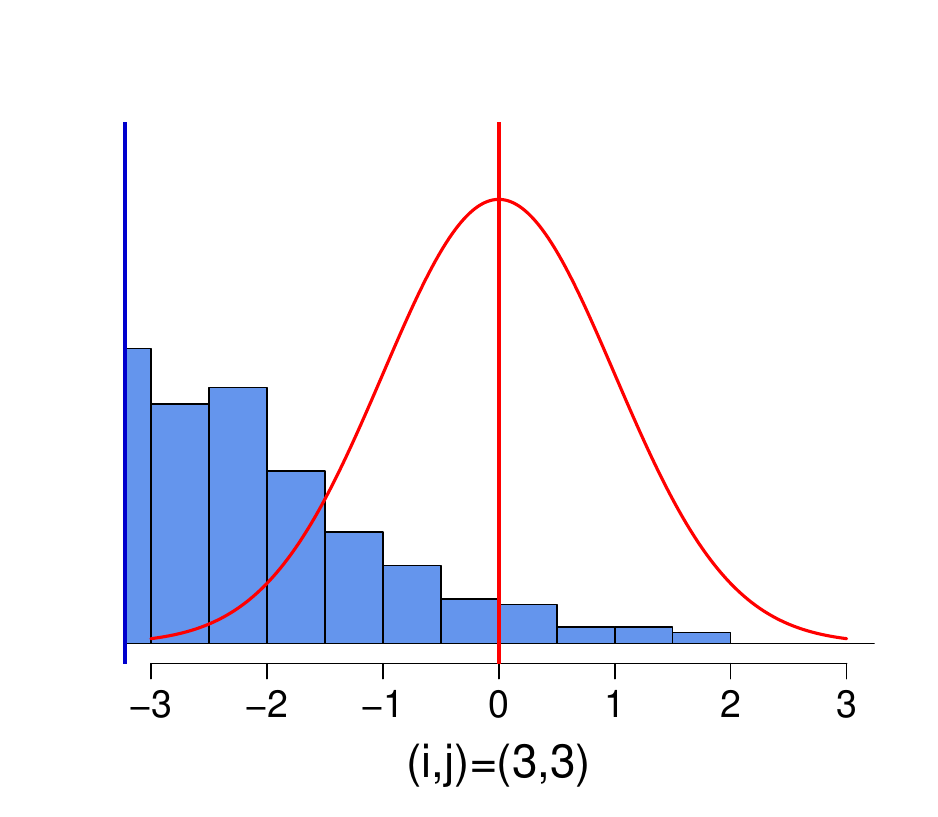}
    \end{minipage}
    \begin{minipage}{0.24\linewidth}
        \centering
        \includegraphics[width=\textwidth]{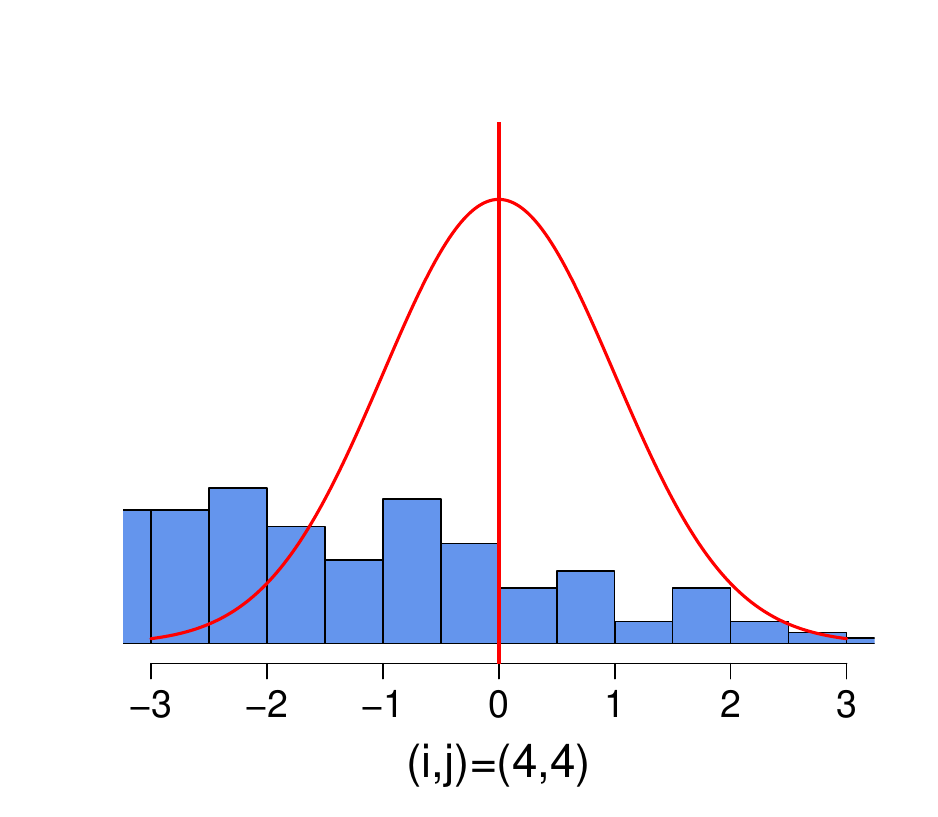}
    \end{minipage}
 \end{minipage}

  \caption*{$n=400, p=200$}
      \vspace{-0.43cm}
 \begin{minipage}{0.3\linewidth}
    \begin{minipage}{0.24\linewidth}
        \centering
        \includegraphics[width=\textwidth]{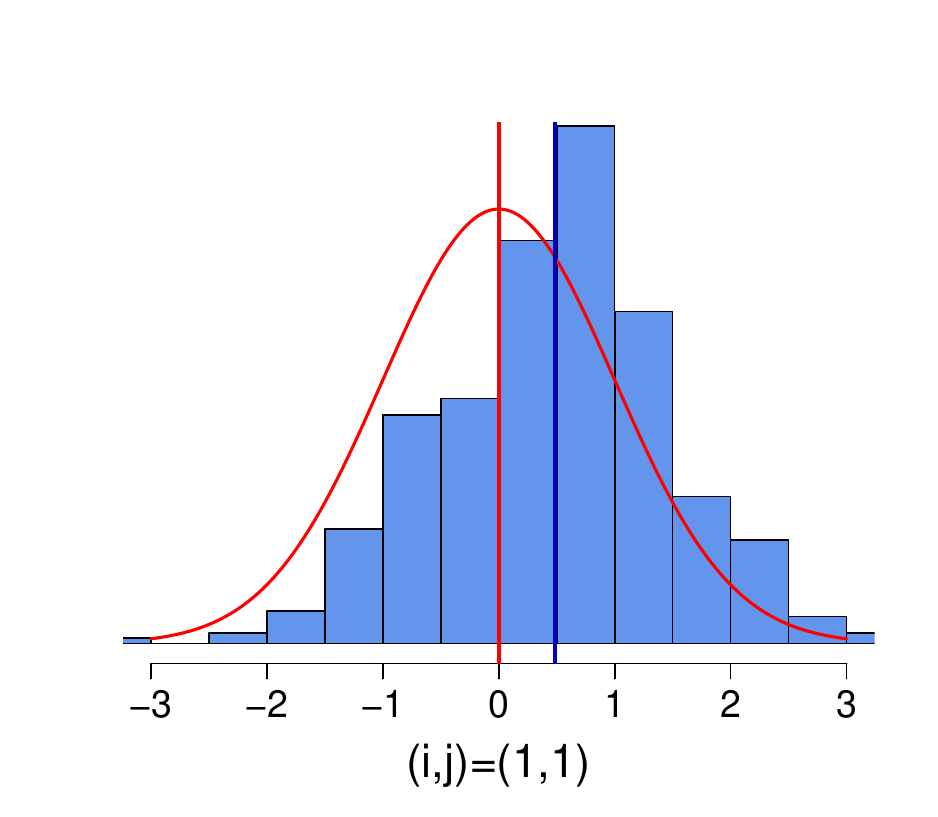}
    \end{minipage}
    \begin{minipage}{0.24\linewidth}
        \centering
        \includegraphics[width=\textwidth]{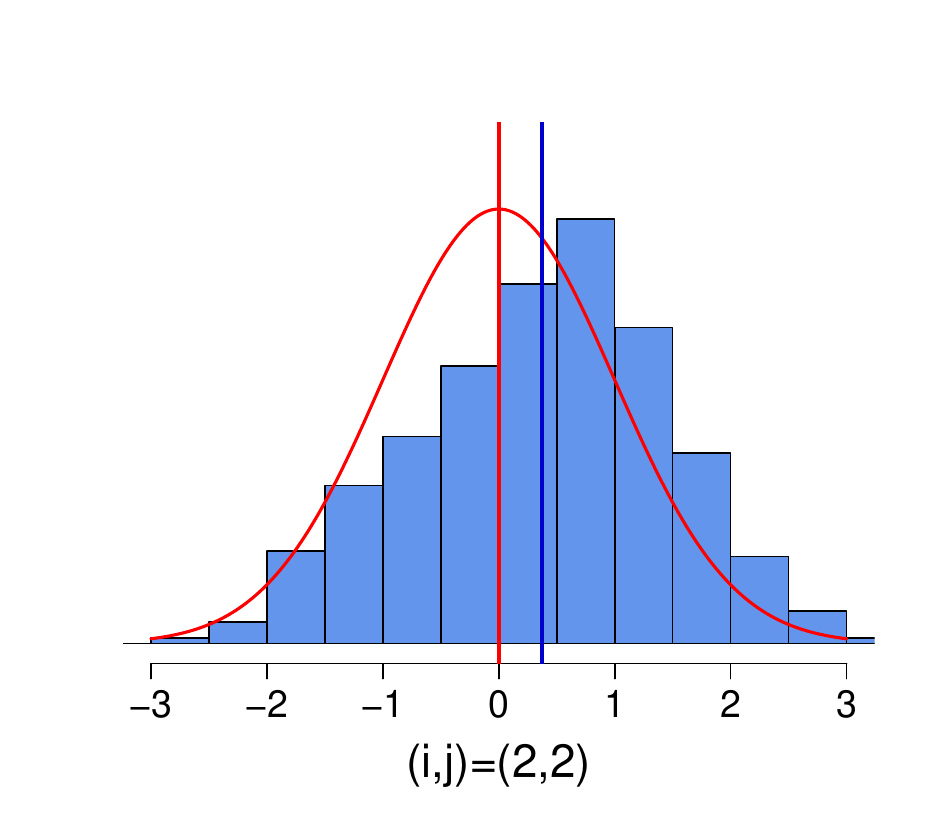}
    \end{minipage}
    \begin{minipage}{0.24\linewidth}
        \centering
        \includegraphics[width=\textwidth]{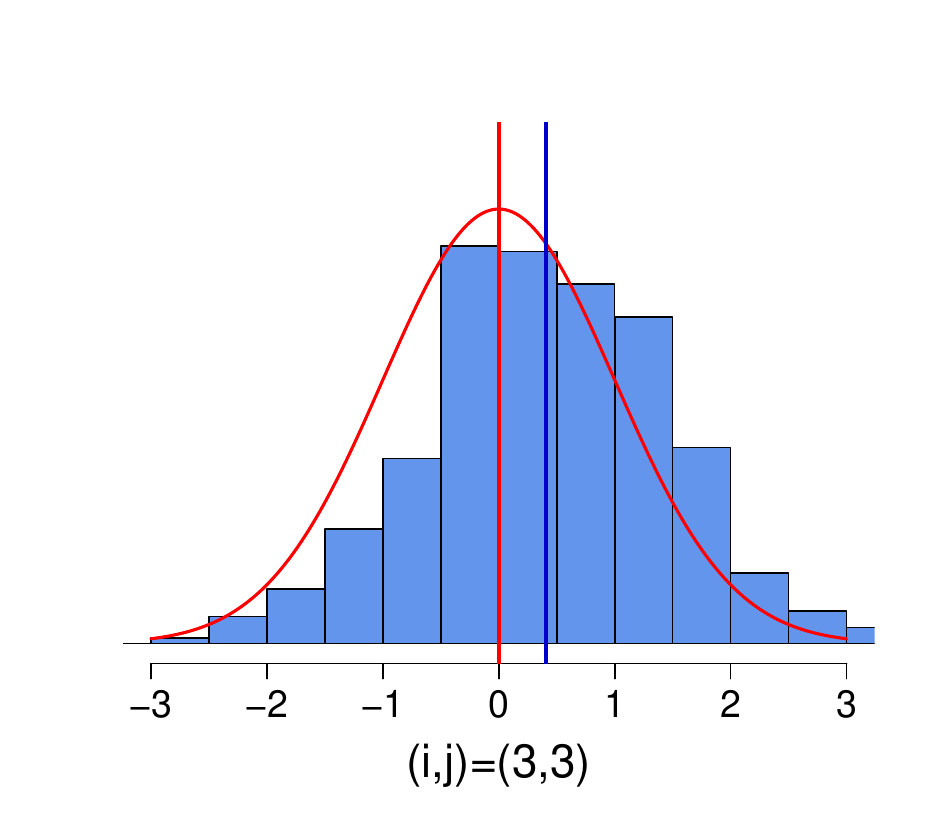}
    \end{minipage}
    \begin{minipage}{0.24\linewidth}
        \centering
        \includegraphics[width=\textwidth]{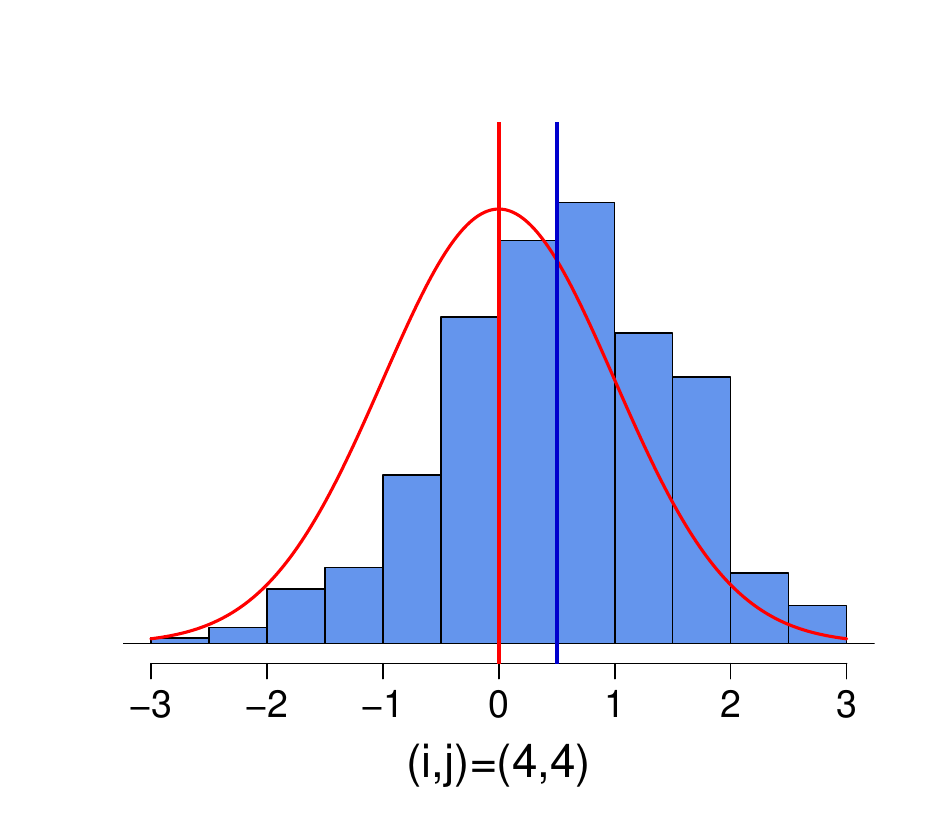}
    \end{minipage}
 \end{minipage}  
     \hspace{1cm}
 \begin{minipage}{0.3\linewidth}
    \begin{minipage}{0.24\linewidth}
        \centering
        \includegraphics[width=\textwidth]{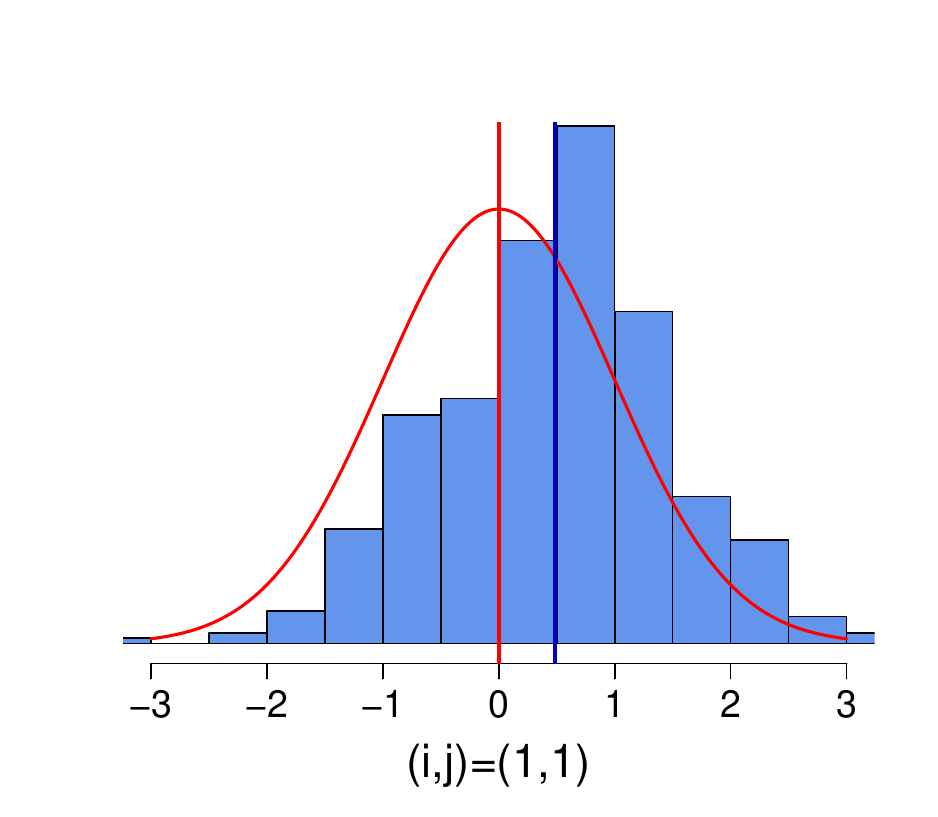}
    \end{minipage}
    \begin{minipage}{0.24\linewidth}
        \centering
        \includegraphics[width=\textwidth]{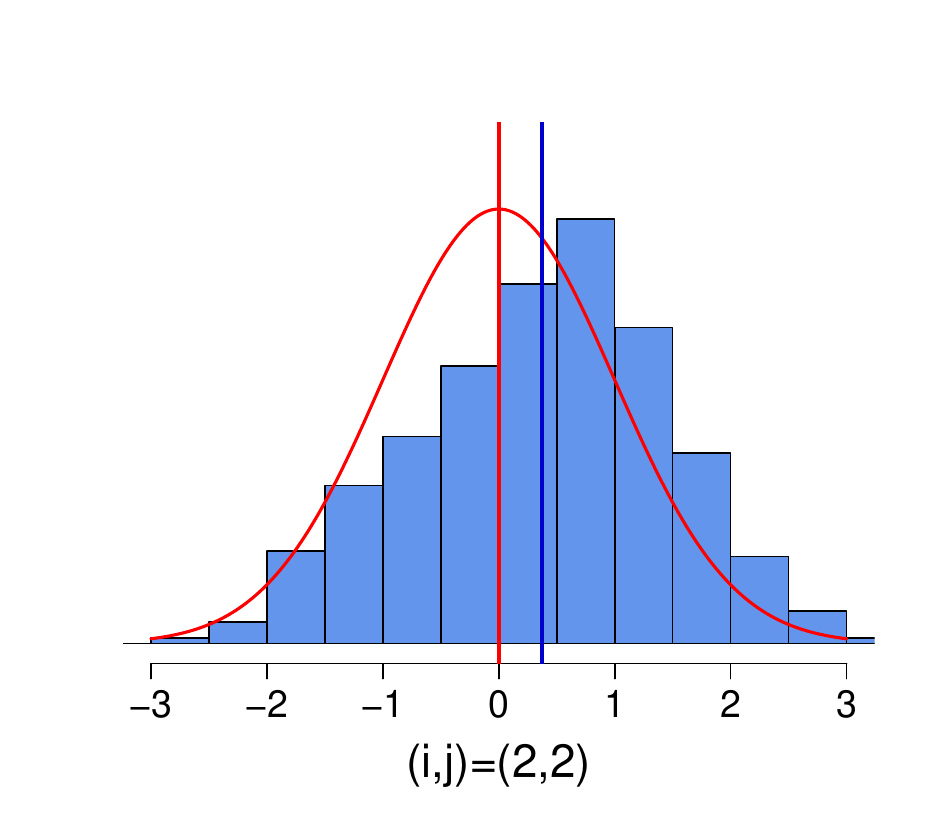}
    \end{minipage}
    \begin{minipage}{0.24\linewidth}
        \centering
        \includegraphics[width=\textwidth]{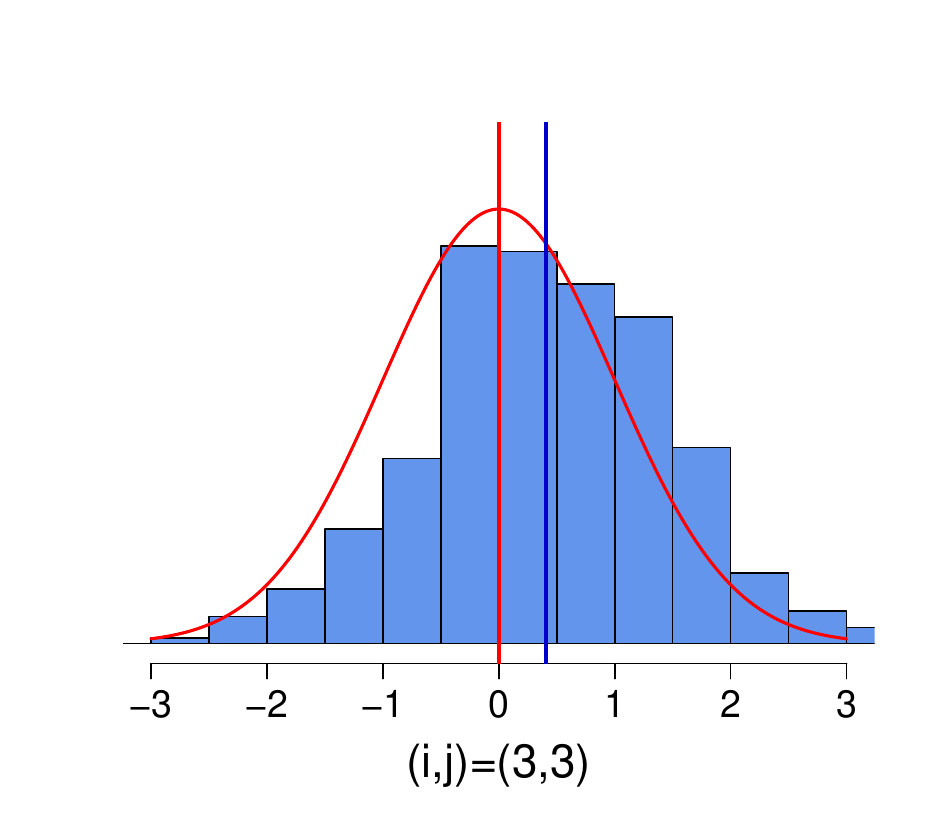}
    \end{minipage}
    \begin{minipage}{0.24\linewidth}
        \centering
        \includegraphics[width=\textwidth]{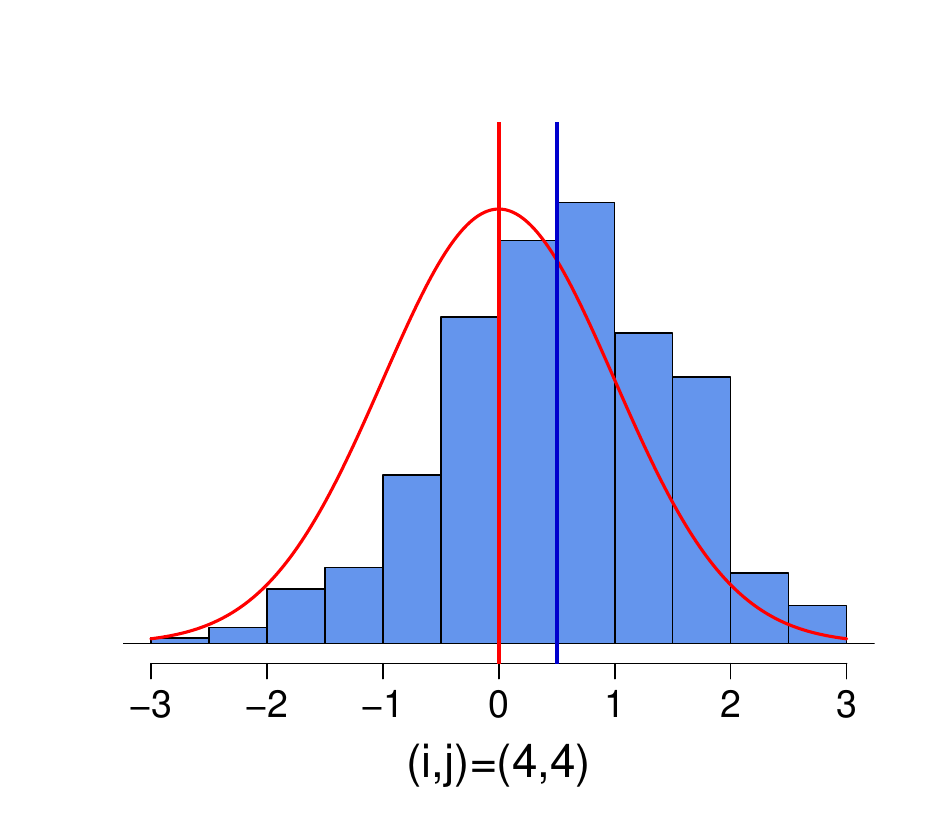}
    \end{minipage}
  \end{minipage}  
    \hspace{1cm}
 \begin{minipage}{0.3\linewidth}
    \begin{minipage}{0.24\linewidth}
        \centering
        \includegraphics[width=\textwidth]{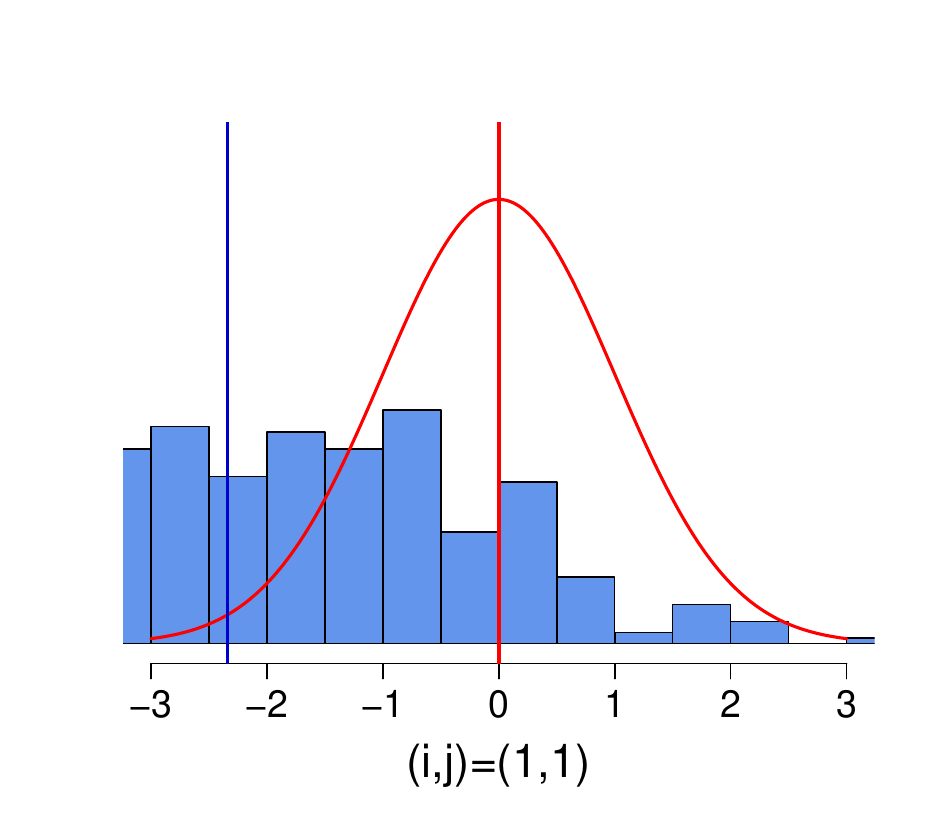}
    \end{minipage}
    \begin{minipage}{0.24\linewidth}
        \centering
        \includegraphics[width=\textwidth]{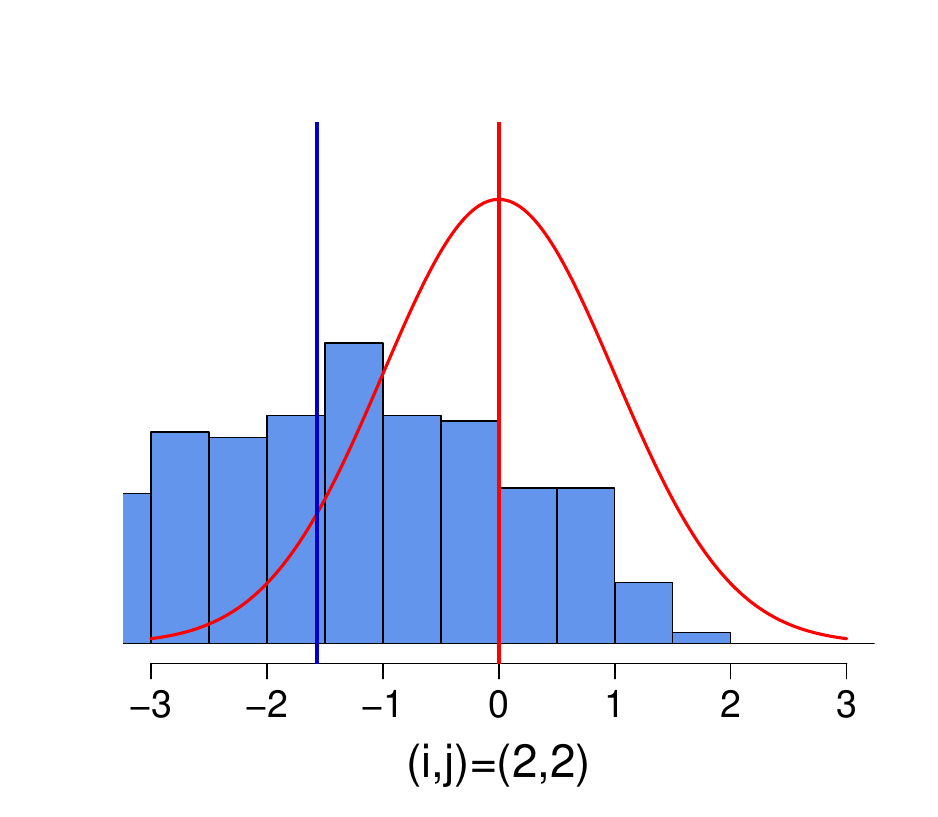}
    \end{minipage}
    \begin{minipage}{0.24\linewidth}
        \centering
        \includegraphics[width=\textwidth]{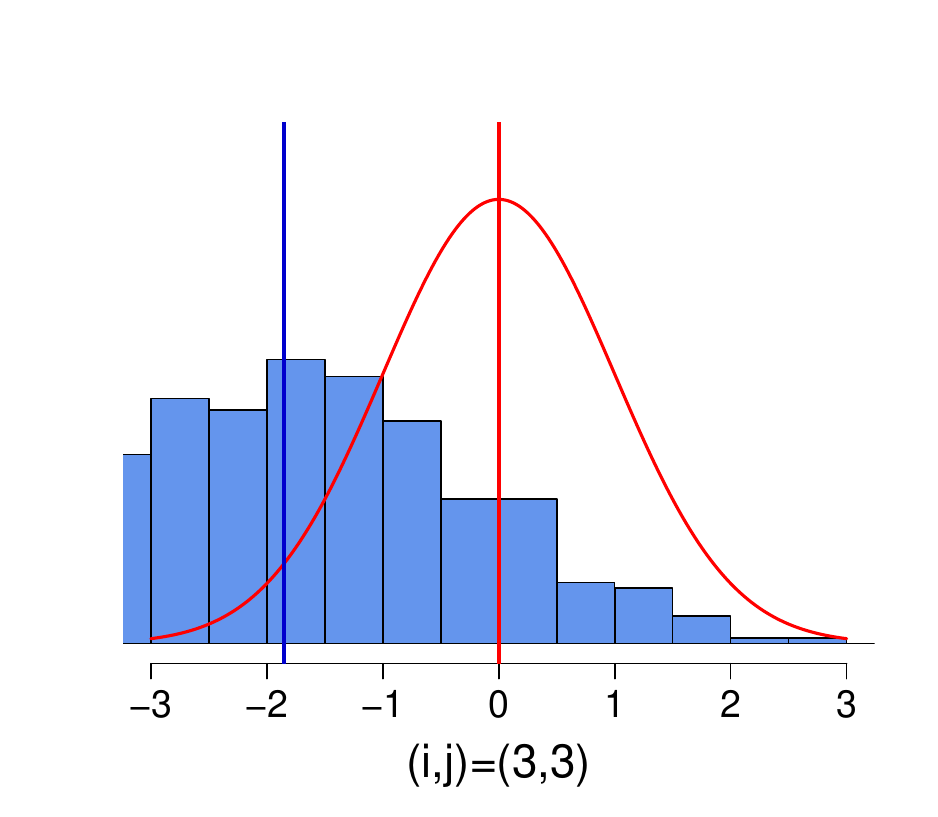}
    \end{minipage}
    \begin{minipage}{0.24\linewidth}
        \centering
        \includegraphics[width=\textwidth]{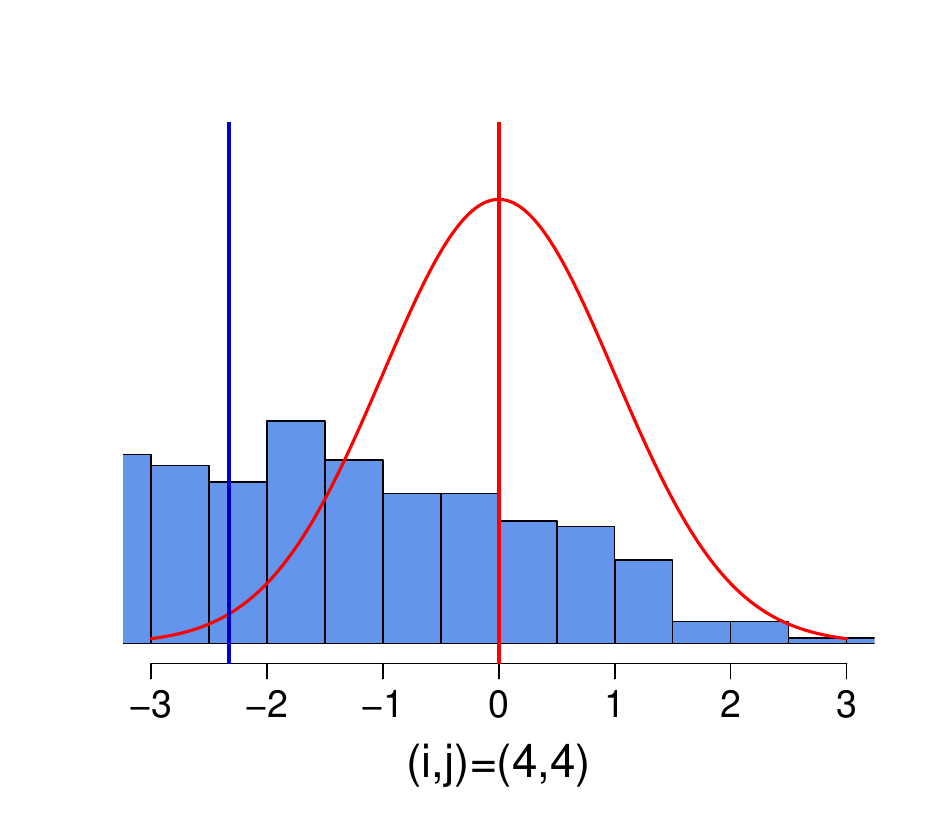}
    \end{minipage}
 \end{minipage}

  \caption*{$n=800, p=200$}
      \vspace{-0.43cm}
 \begin{minipage}{0.3\linewidth}
    \begin{minipage}{0.24\linewidth}
        \centering
        \includegraphics[width=\textwidth]{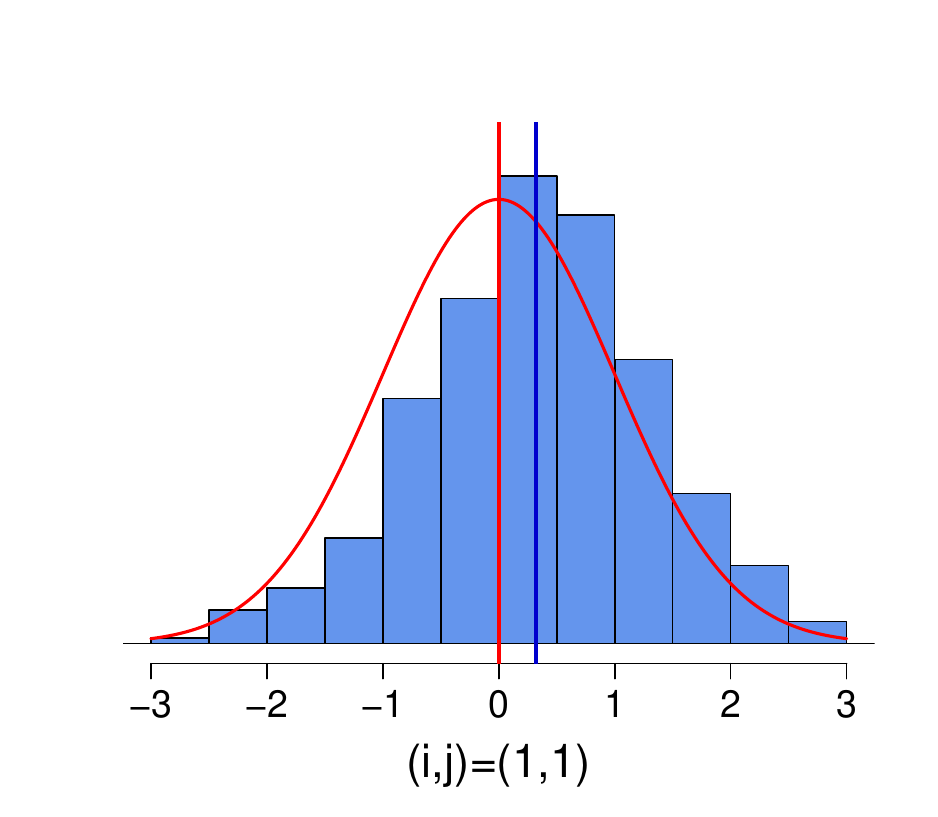}
    \end{minipage}
    \begin{minipage}{0.24\linewidth}
        \centering
        \includegraphics[width=\textwidth]{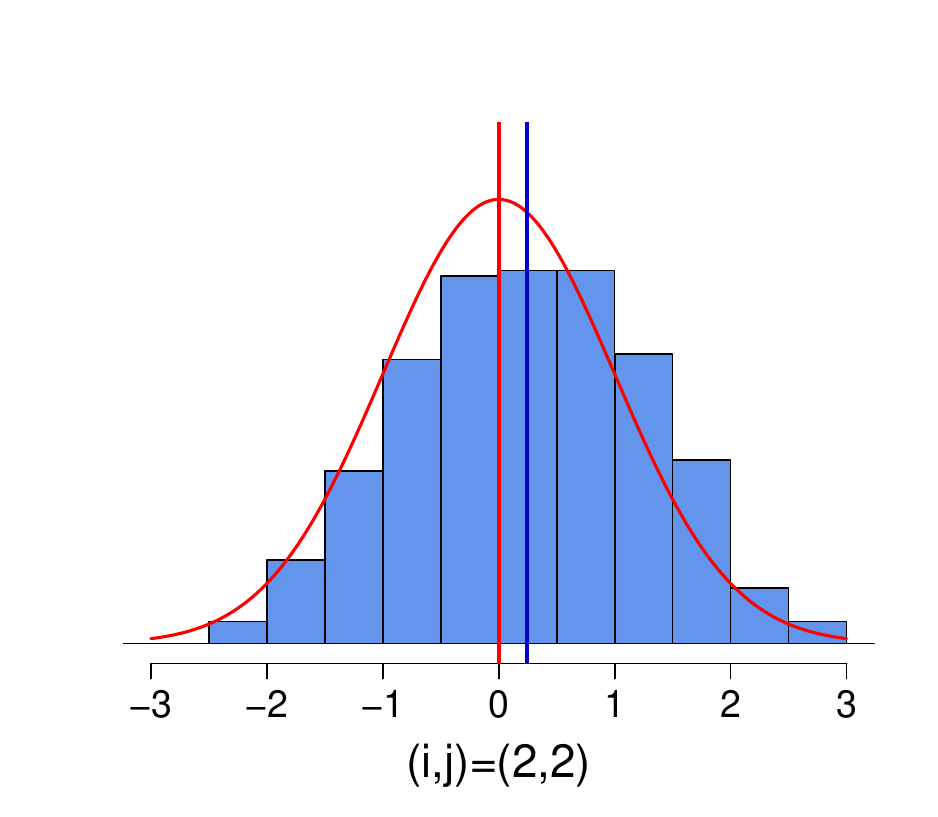}
    \end{minipage}
    \begin{minipage}{0.24\linewidth}
        \centering
        \includegraphics[width=\textwidth]{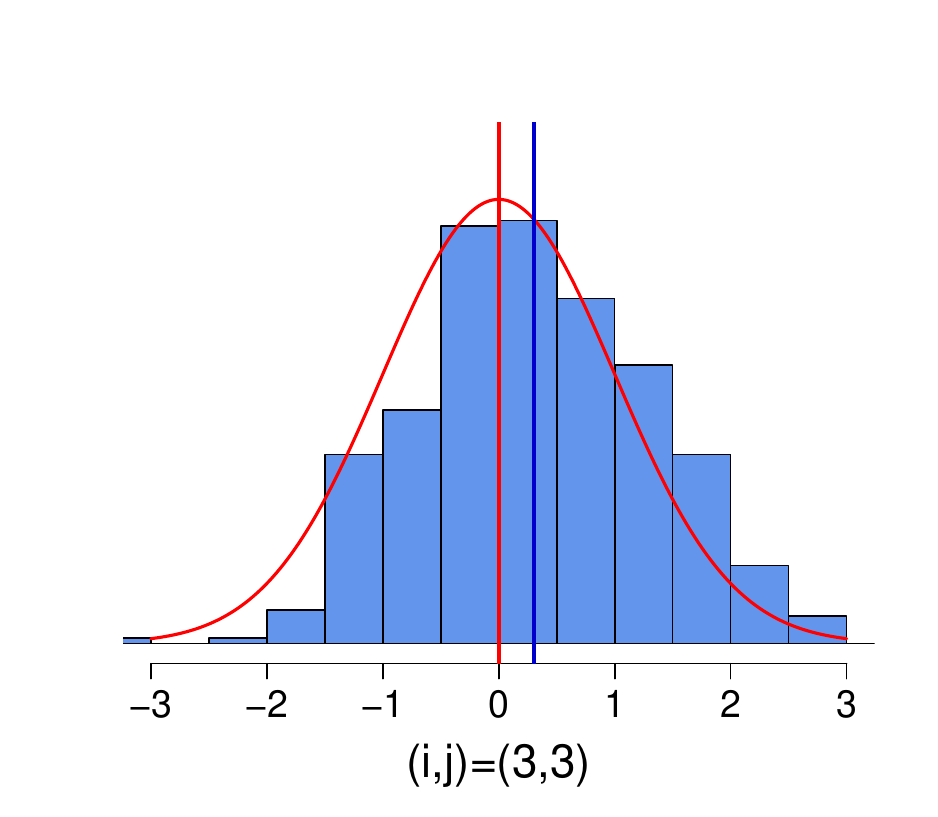}
    \end{minipage}
    \begin{minipage}{0.24\linewidth}
        \centering
        \includegraphics[width=\textwidth]{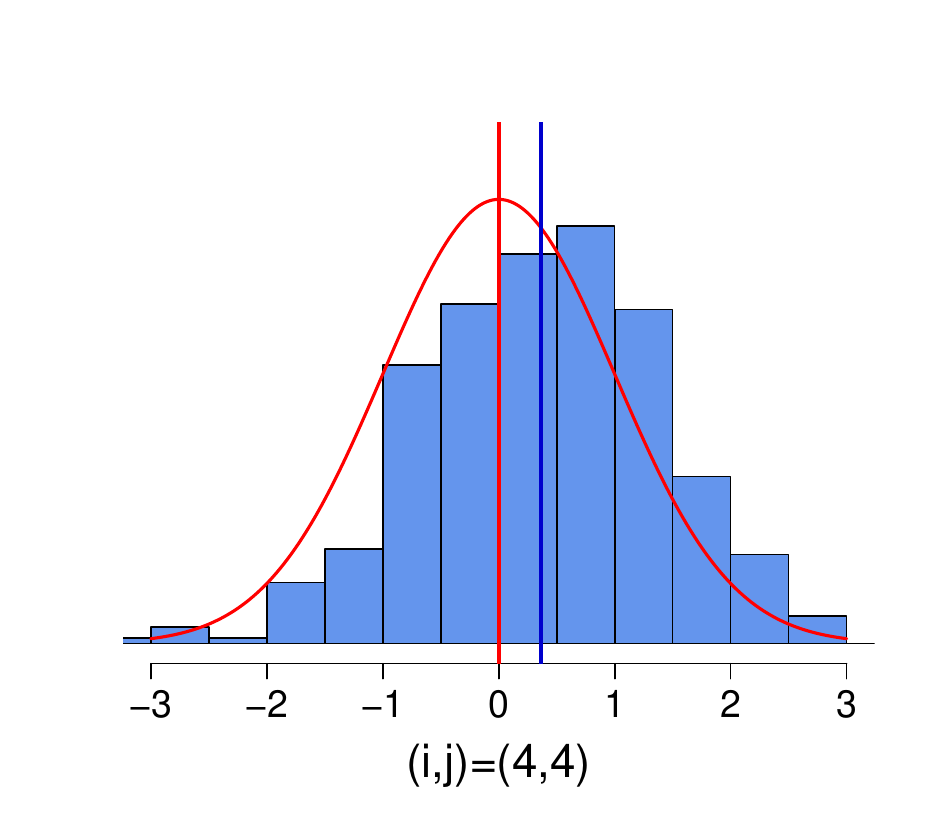}
    \end{minipage}
 \end{minipage} 
     \hspace{1cm}
 \begin{minipage}{0.3\linewidth}
    \begin{minipage}{0.24\linewidth}
        \centering
        \includegraphics[width=\textwidth]{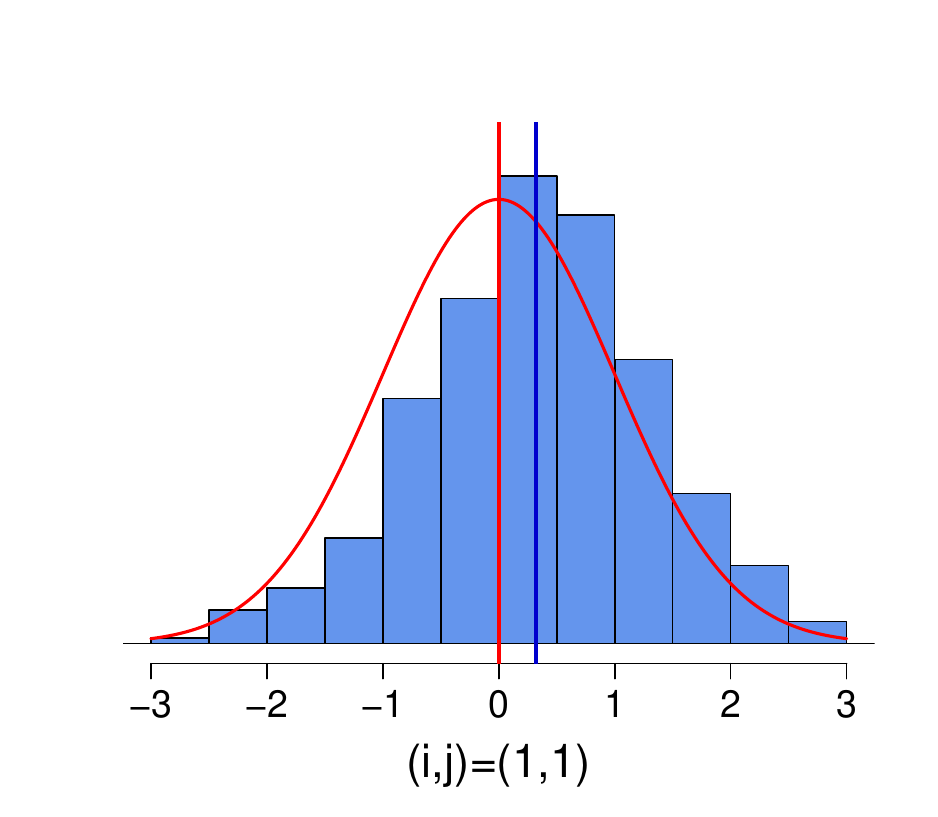}
    \end{minipage}
    \begin{minipage}{0.24\linewidth}
        \centering
        \includegraphics[width=\textwidth]{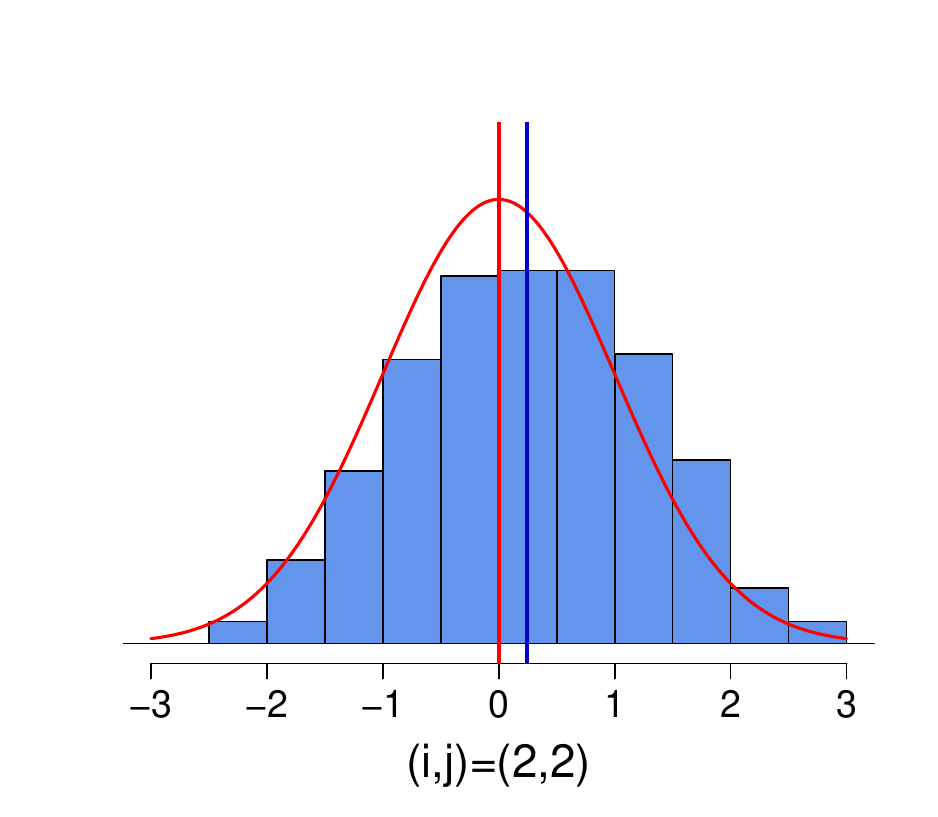}
    \end{minipage}
    \begin{minipage}{0.24\linewidth}
        \centering
        \includegraphics[width=\textwidth]{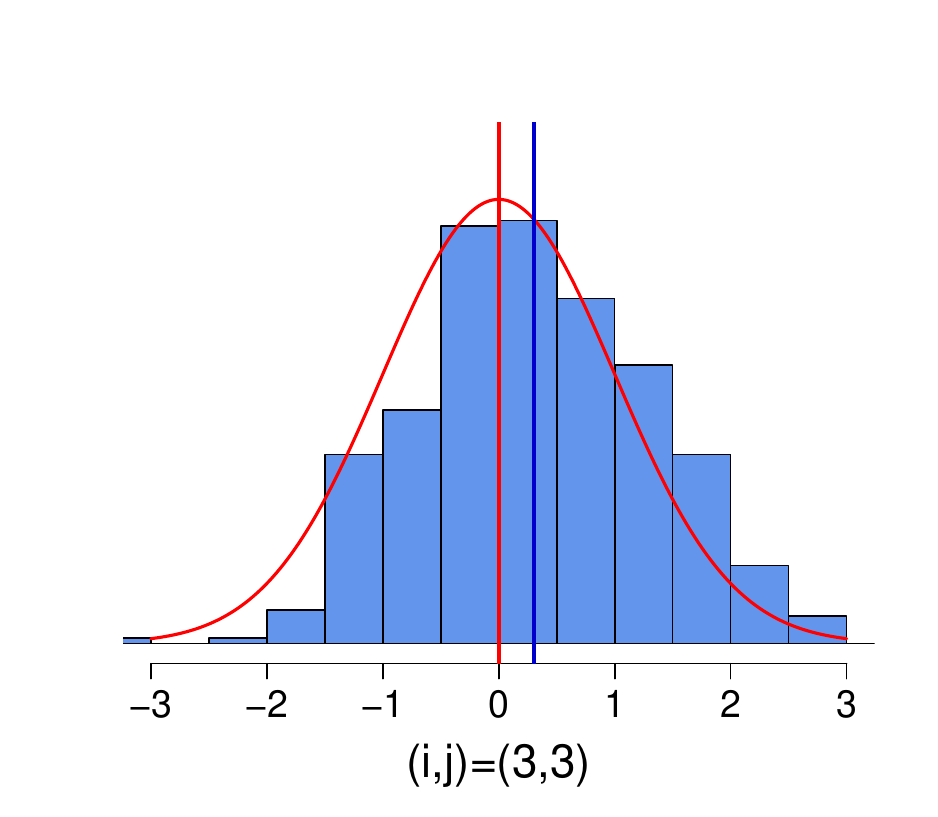}
    \end{minipage}
    \begin{minipage}{0.24\linewidth}
        \centering
        \includegraphics[width=\textwidth]{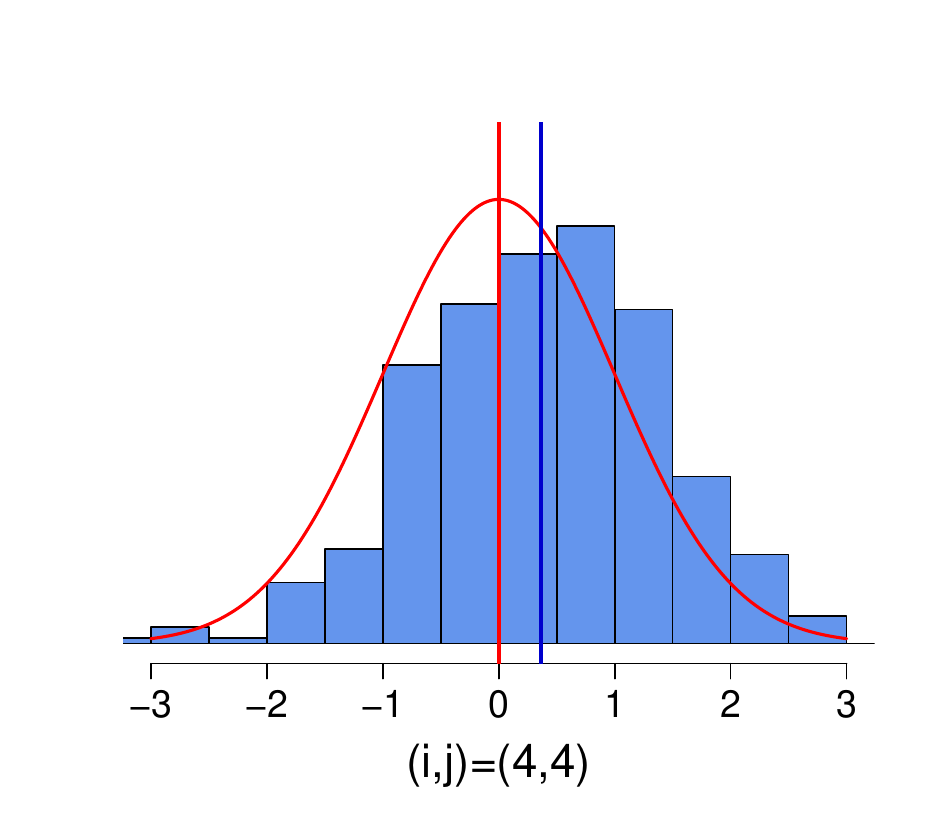}
    \end{minipage}
 \end{minipage}   
      \hspace{1cm}
 \begin{minipage}{0.3\linewidth}
    \begin{minipage}{0.24\linewidth}
        \centering
        \includegraphics[width=\textwidth]{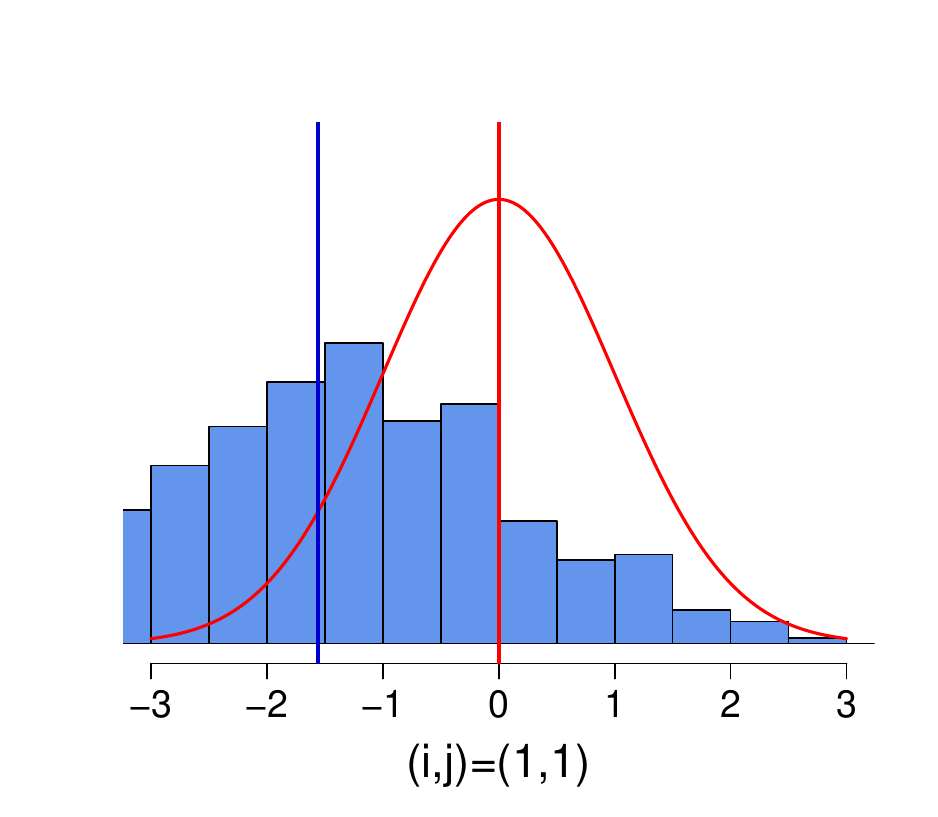}
    \end{minipage}
    \begin{minipage}{0.24\linewidth}
        \centering
        \includegraphics[width=\textwidth]{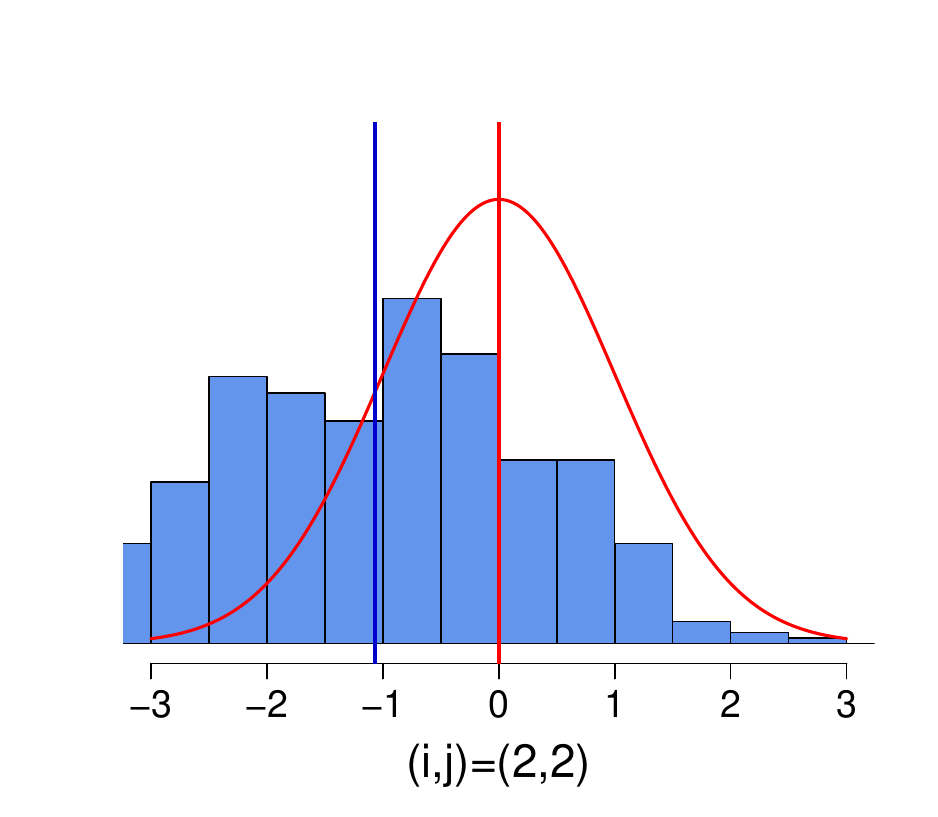}
    \end{minipage}
    \begin{minipage}{0.24\linewidth}
        \centering
        \includegraphics[width=\textwidth]{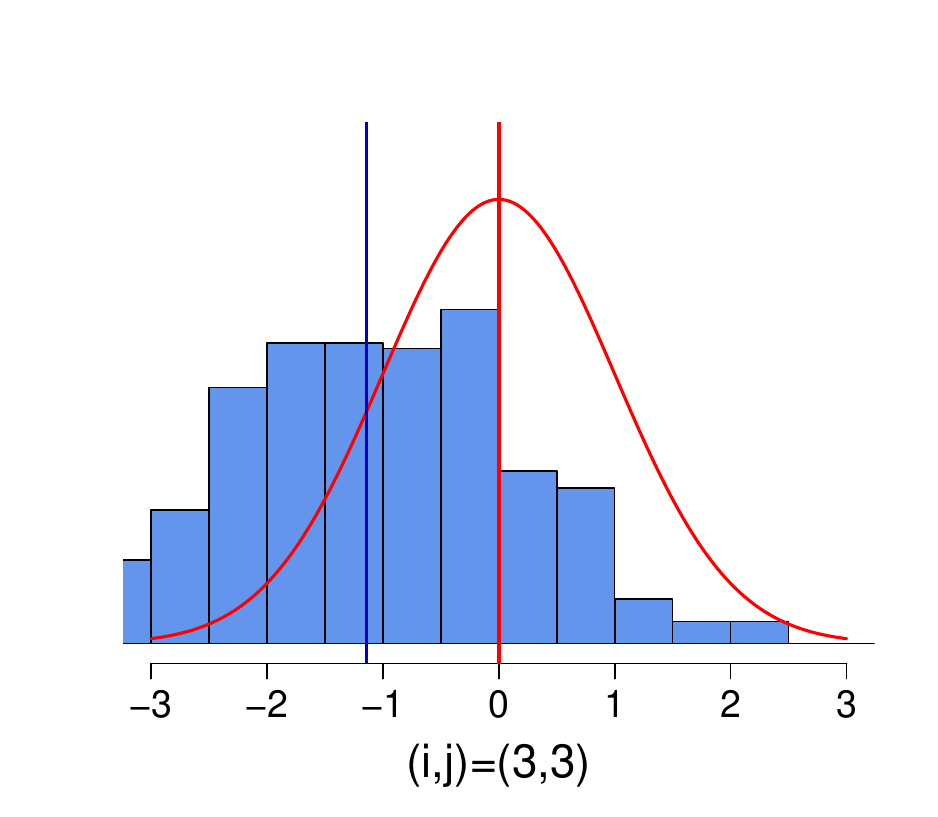}
    \end{minipage}
    \begin{minipage}{0.24\linewidth}
        \centering
        \includegraphics[width=\textwidth]{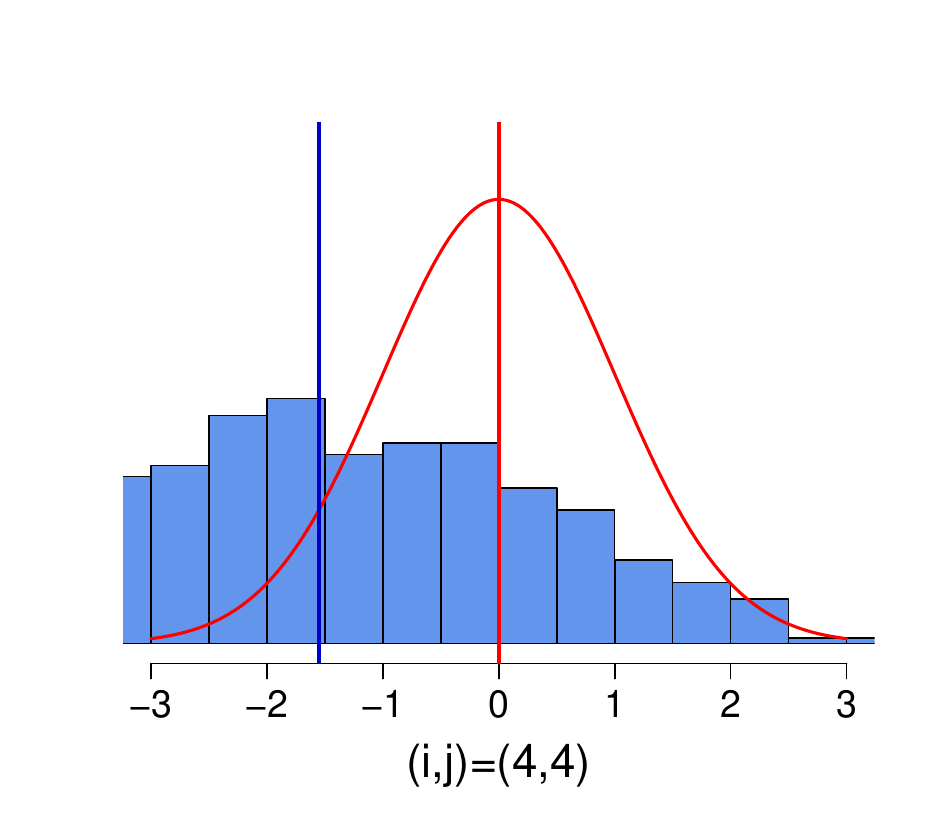}
    \end{minipage}
     \end{minipage}   
     
 \caption*{$n=200, p=400$}
     \vspace{-0.43cm}
 \begin{minipage}{0.3\linewidth}
    \begin{minipage}{0.24\linewidth}
        \centering
        \includegraphics[width=\textwidth]{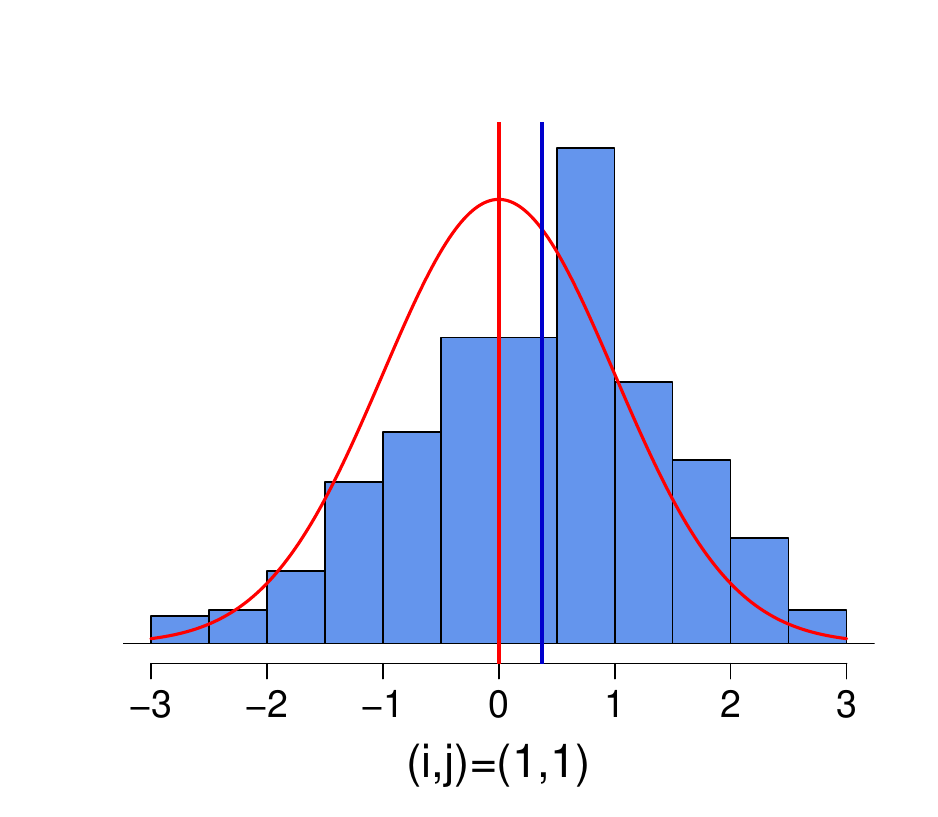}
    \end{minipage}
    \begin{minipage}{0.24\linewidth}
        \centering
        \includegraphics[width=\textwidth]{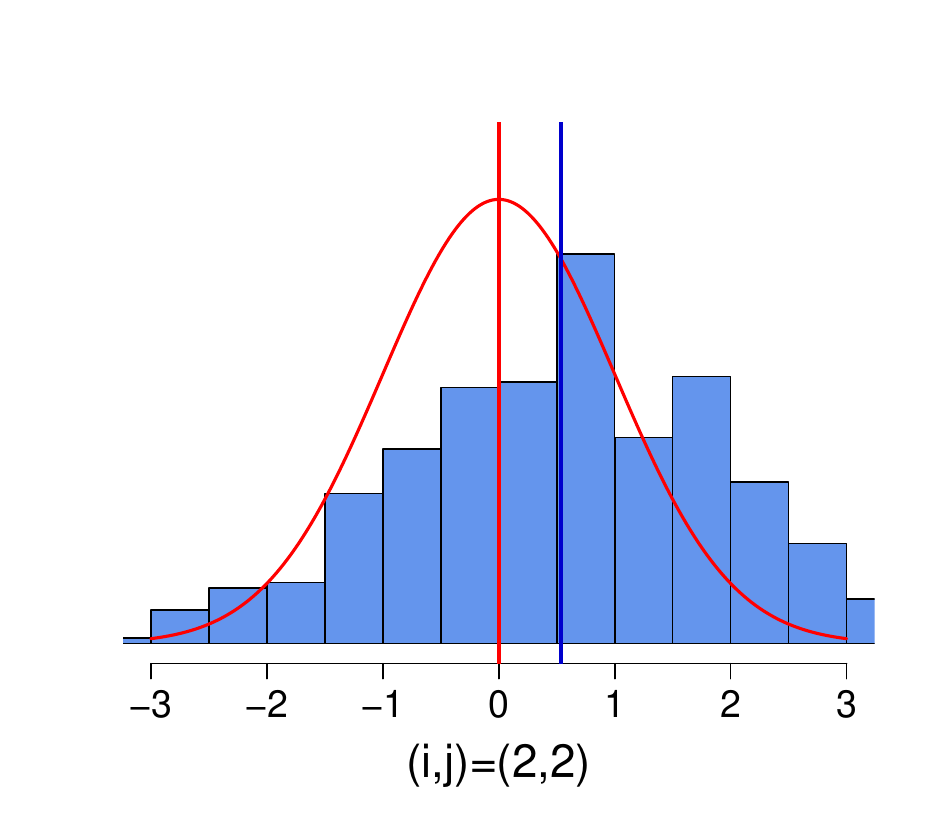}
    \end{minipage}
    \begin{minipage}{0.24\linewidth}
        \centering
        \includegraphics[width=\textwidth]{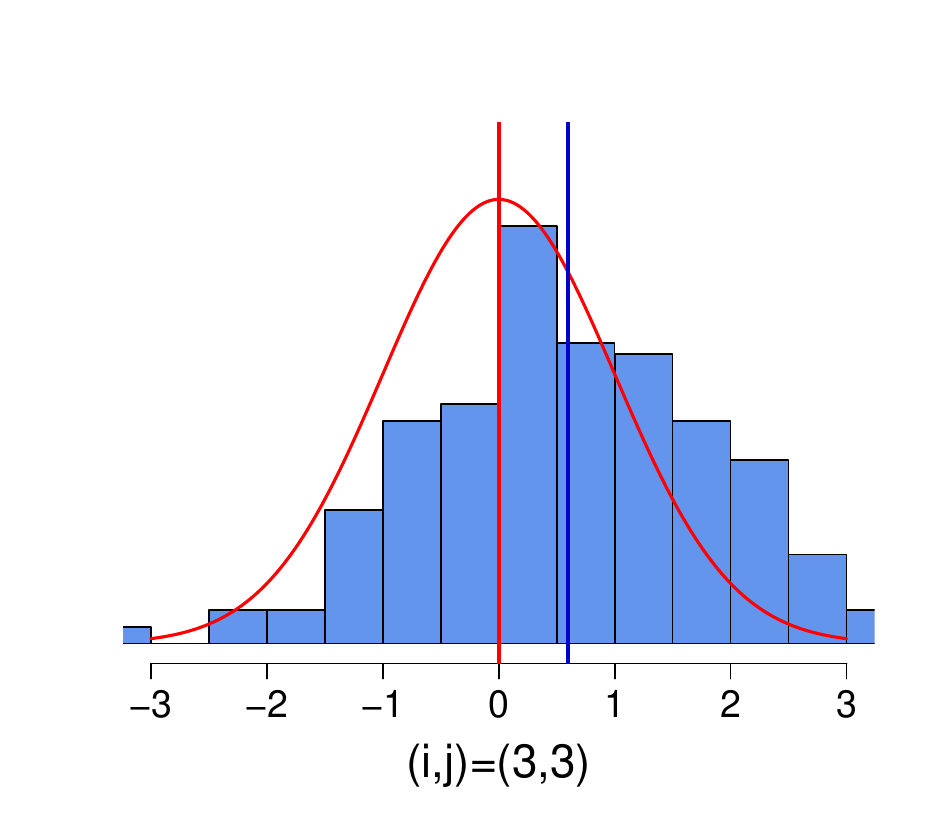}
    \end{minipage}
    \begin{minipage}{0.24\linewidth}
        \centering
        \includegraphics[width=\textwidth]{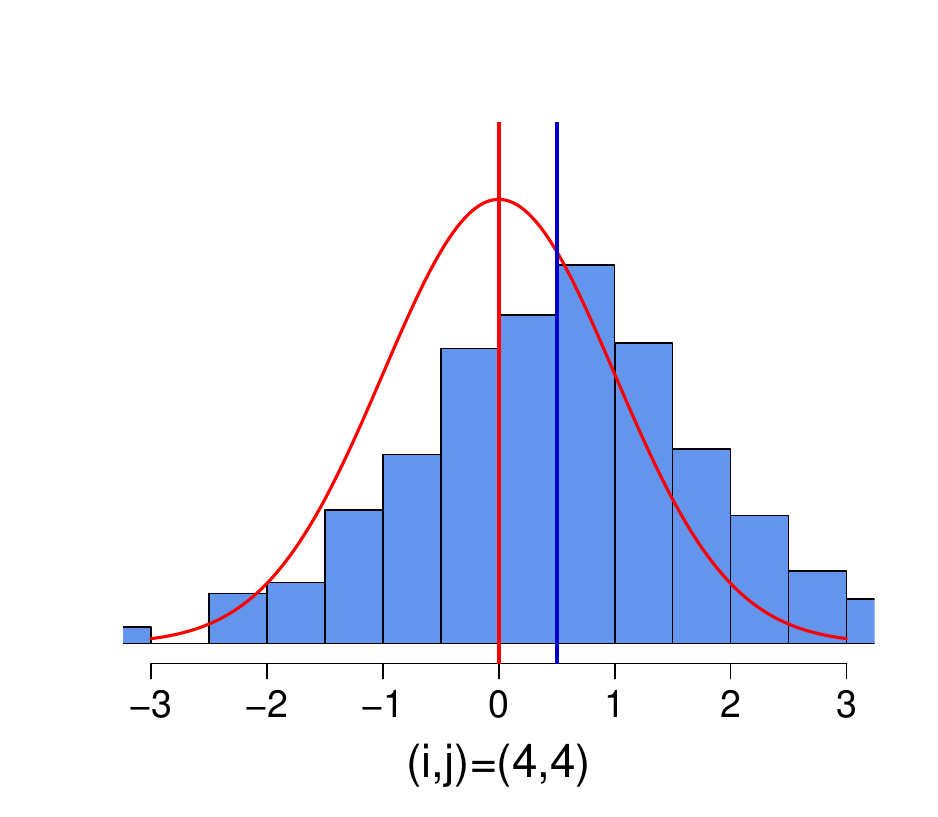}
    \end{minipage}
 \end{minipage}
 \hspace{1cm}
 \begin{minipage}{0.3\linewidth}
    \begin{minipage}{0.24\linewidth}
        \centering
        \includegraphics[width=\textwidth]{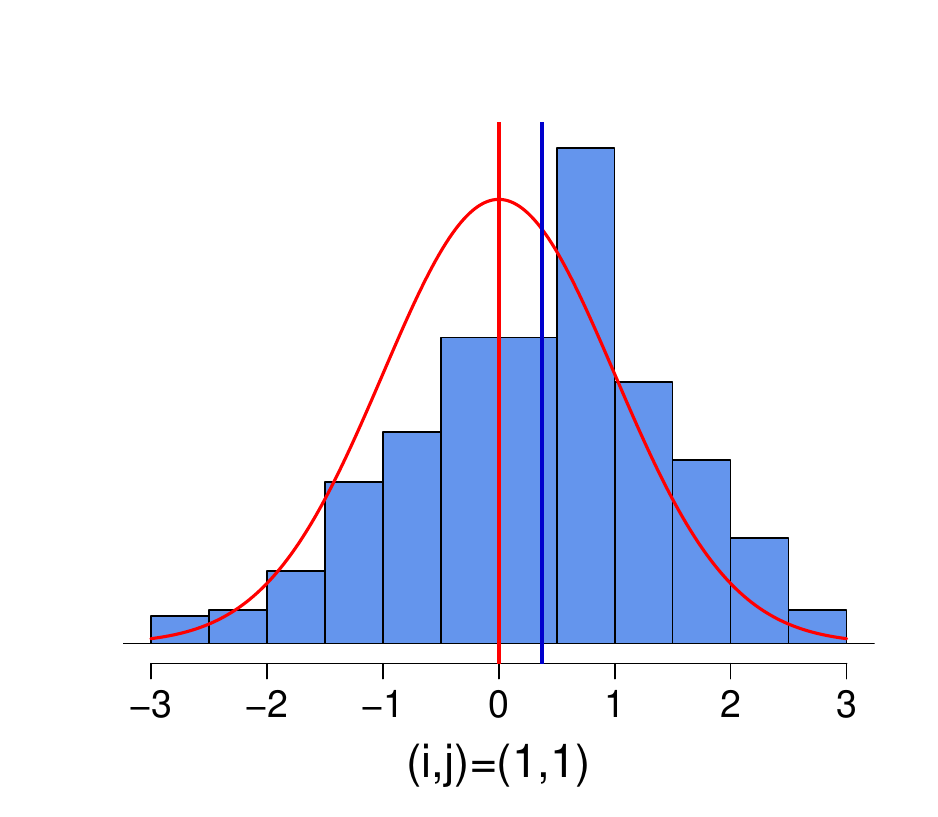}
    \end{minipage}
    \begin{minipage}{0.24\linewidth}
        \centering
        \includegraphics[width=\textwidth]{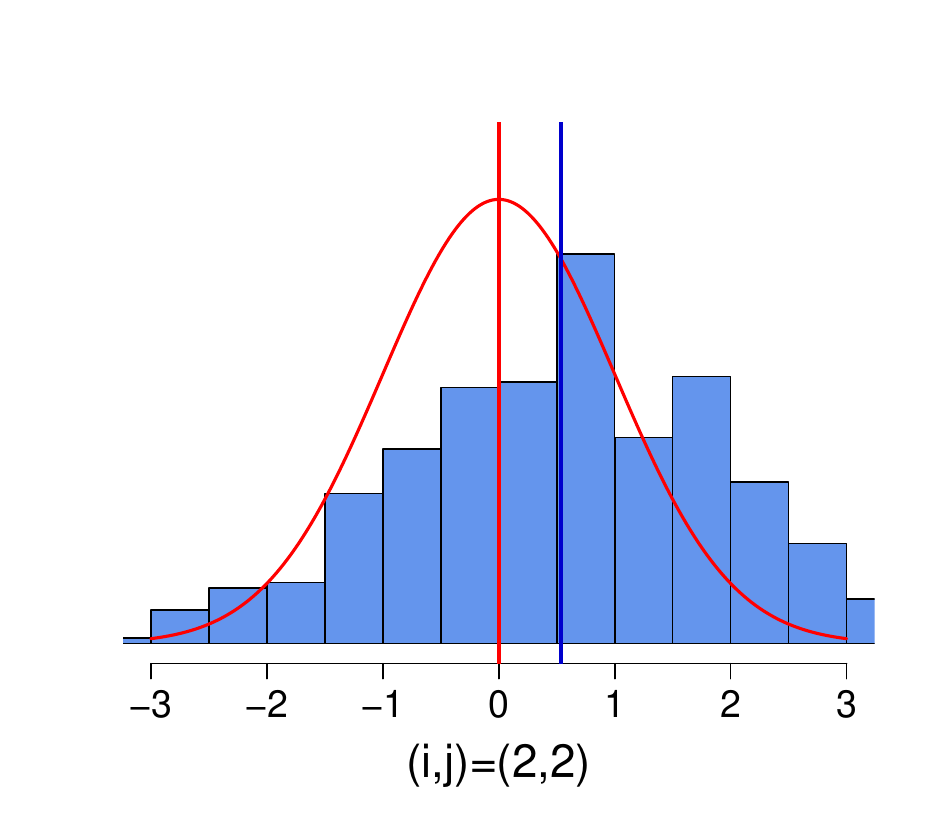}
    \end{minipage}
    \begin{minipage}{0.24\linewidth}
        \centering
        \includegraphics[width=\textwidth]{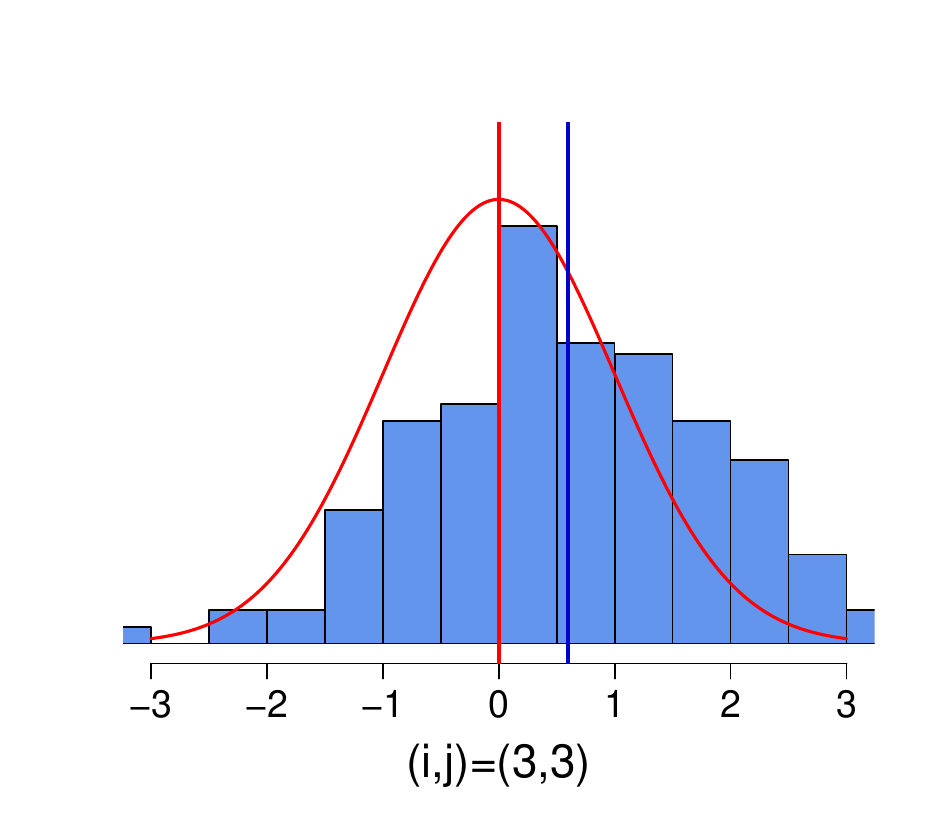}
    \end{minipage}
    \begin{minipage}{0.24\linewidth}
        \centering
        \includegraphics[width=\textwidth]{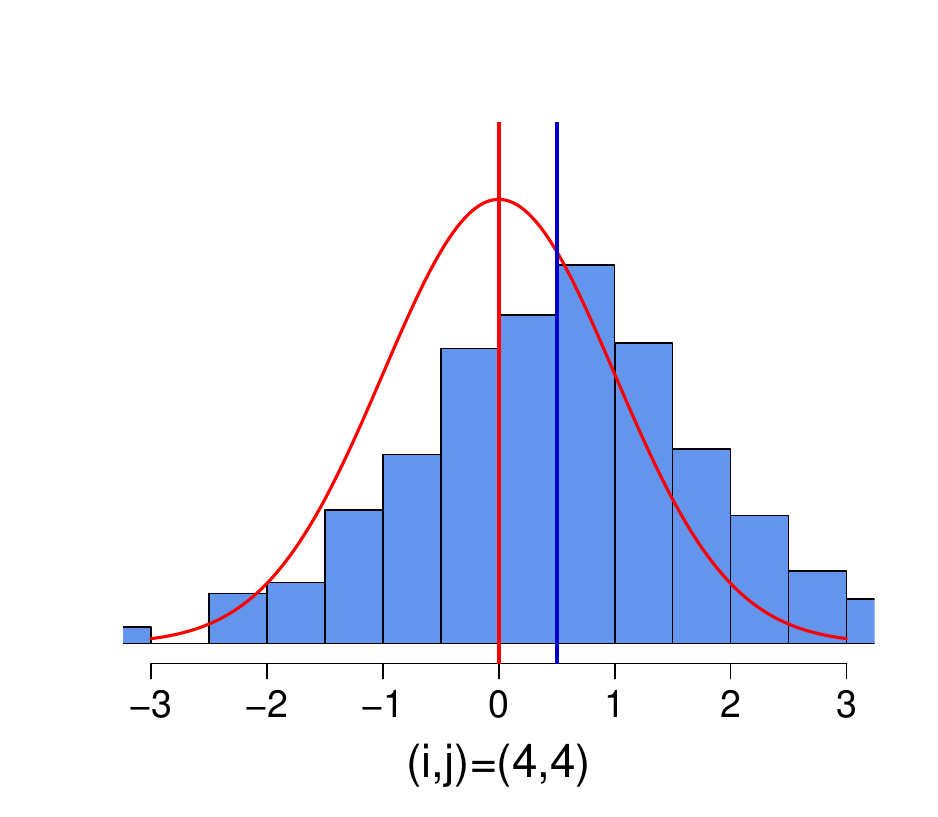}
    \end{minipage}    
 \end{minipage}
  \hspace{1cm}
 \begin{minipage}{0.3\linewidth}
     \begin{minipage}{0.24\linewidth}
        \centering
        \includegraphics[width=\textwidth]{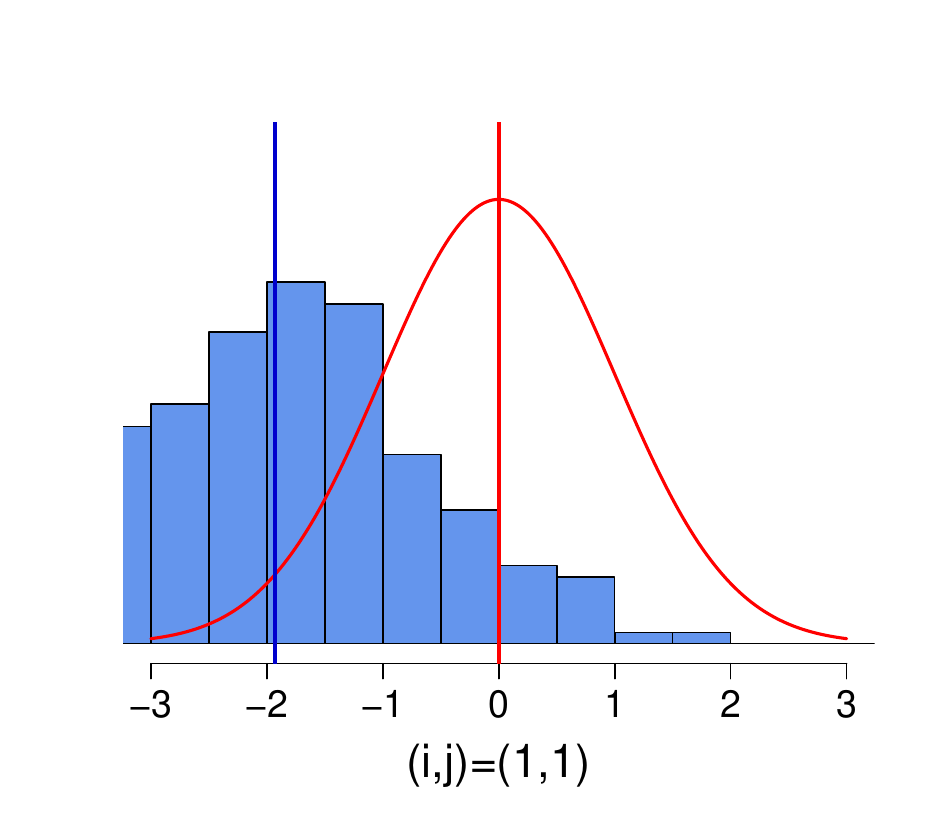}
    \end{minipage}
    \begin{minipage}{0.24\linewidth}
        \centering
        \includegraphics[width=\textwidth]{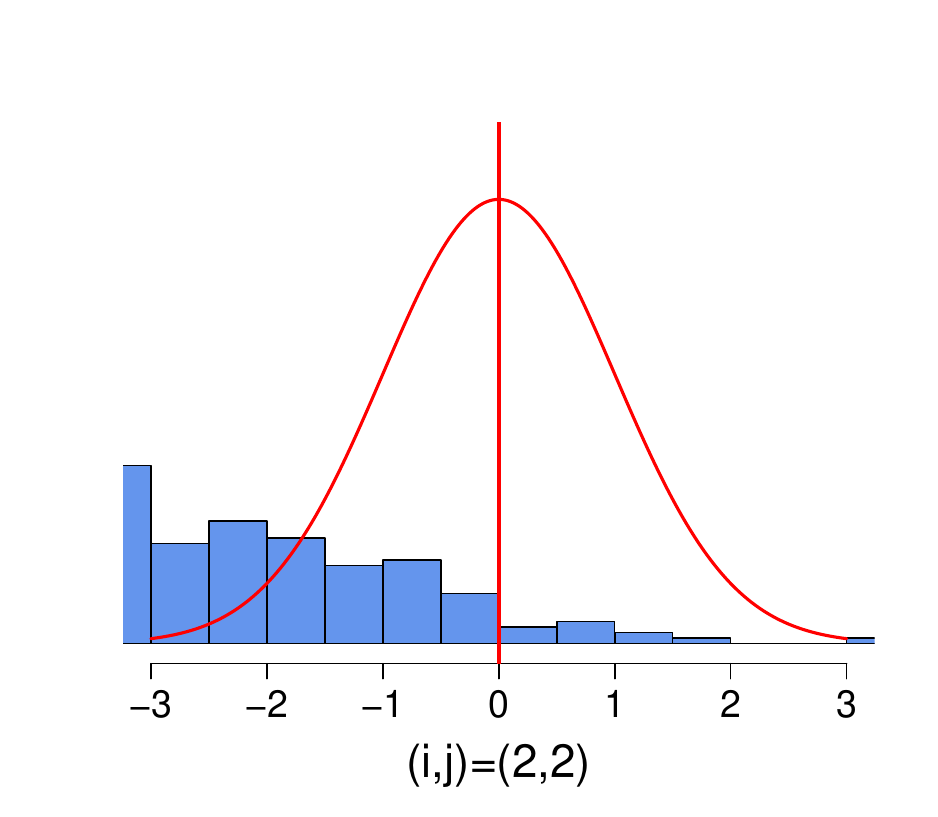}
    \end{minipage}
    \begin{minipage}{0.24\linewidth}
        \centering
        \includegraphics[width=\textwidth]{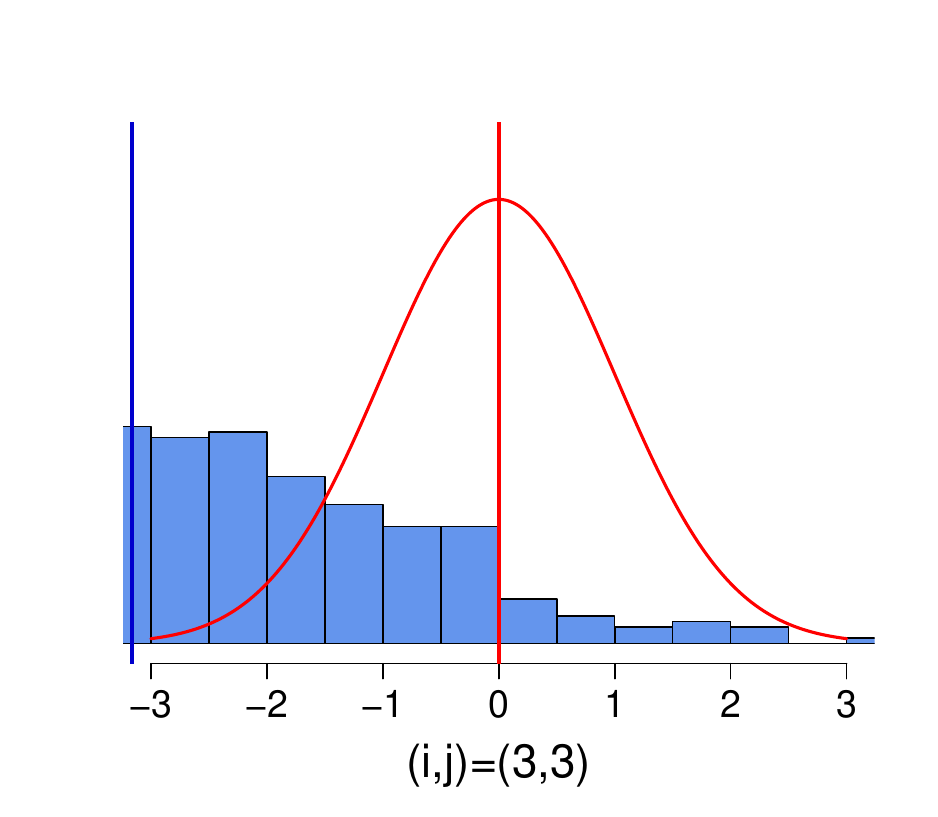}
    \end{minipage}
    \begin{minipage}{0.24\linewidth}
        \centering
        \includegraphics[width=\textwidth]{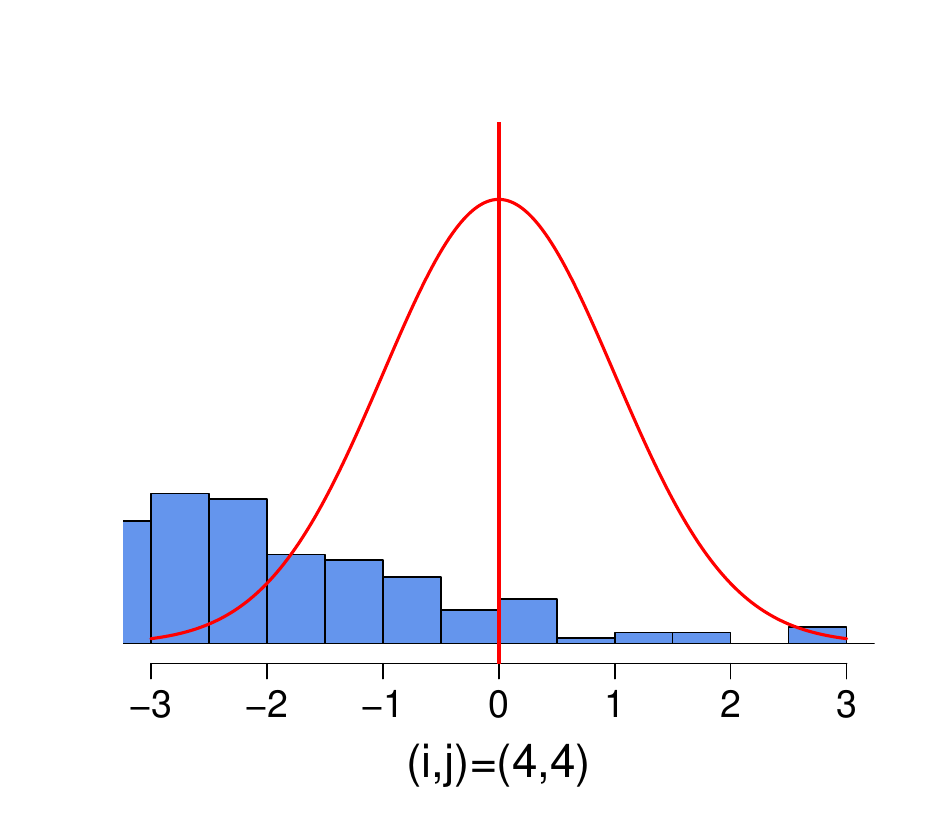}
    \end{minipage}
 \end{minipage}

  \caption*{$n=400, p=400$}
      \vspace{-0.43cm}
 \begin{minipage}{0.3\linewidth}
    \begin{minipage}{0.24\linewidth}
        \centering
        \includegraphics[width=\textwidth]{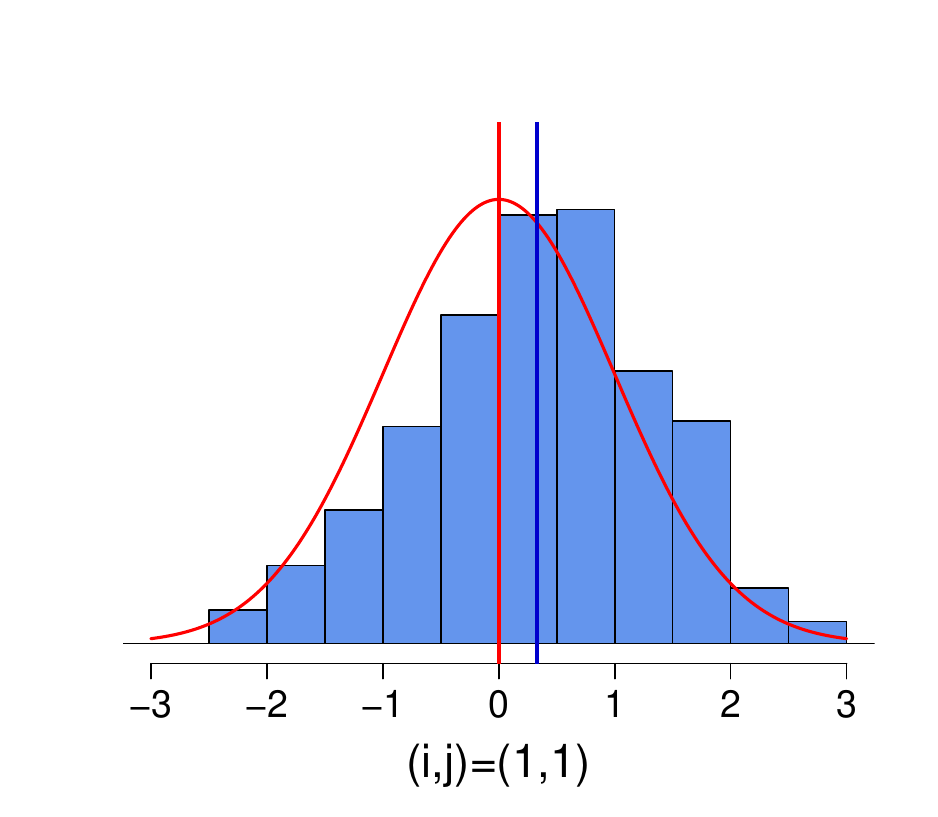}
    \end{minipage}
    \begin{minipage}{0.24\linewidth}
        \centering
        \includegraphics[width=\textwidth]{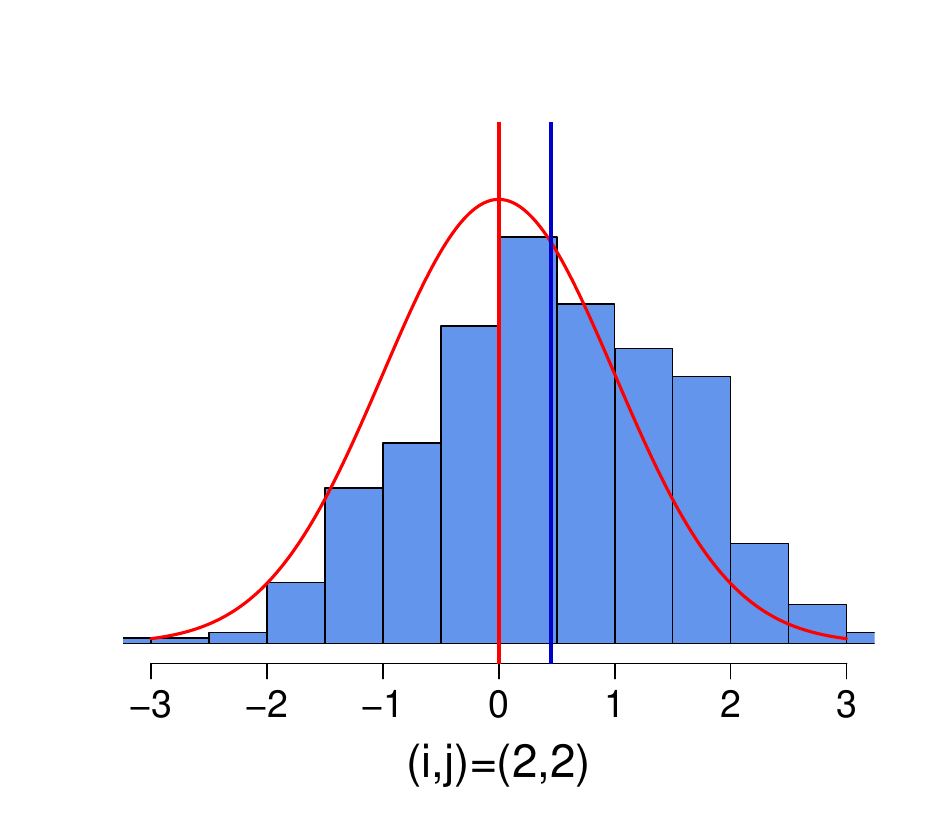}
    \end{minipage}
    \begin{minipage}{0.24\linewidth}
        \centering
        \includegraphics[width=\textwidth]{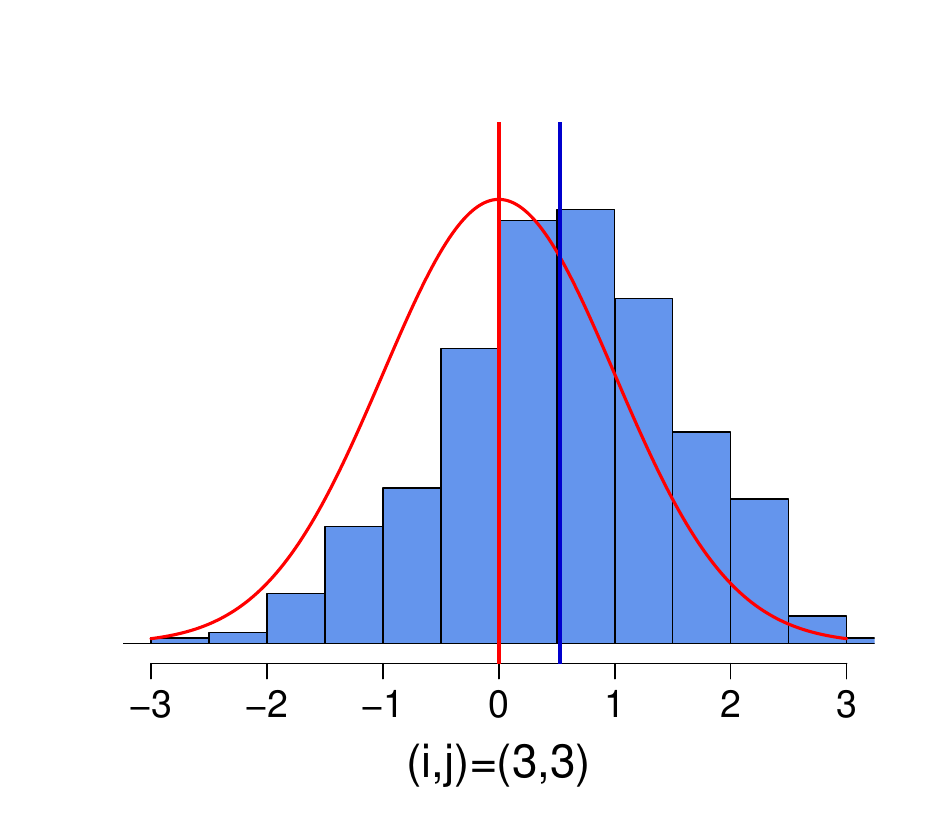}
    \end{minipage}
    \begin{minipage}{0.24\linewidth}
        \centering
        \includegraphics[width=\textwidth]{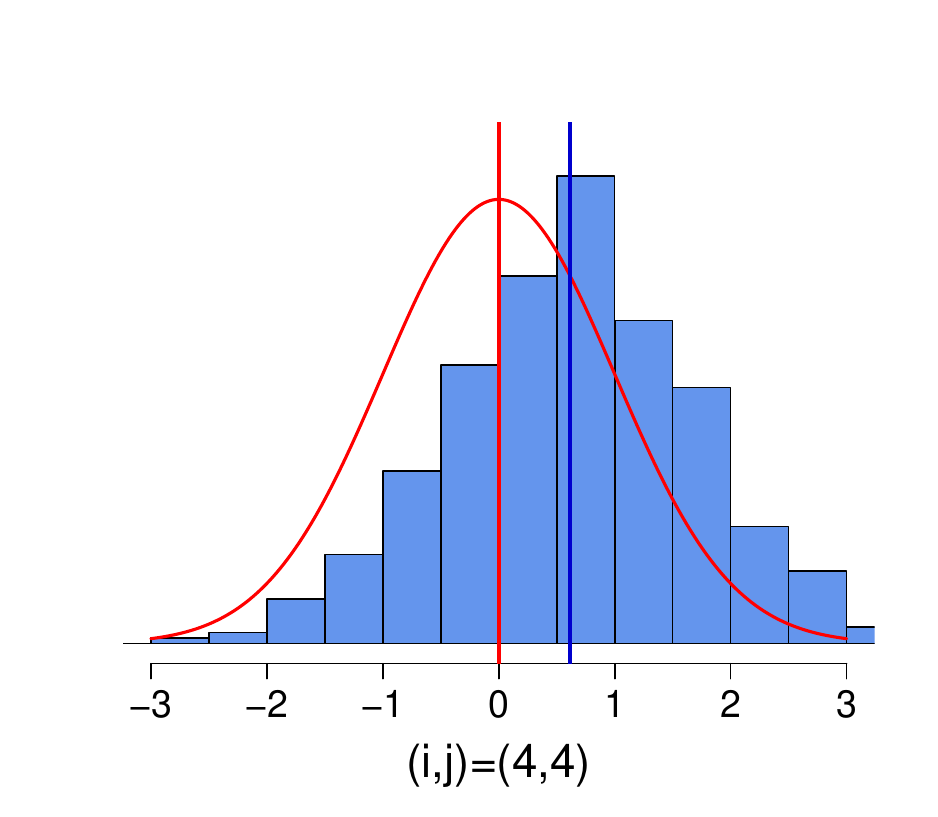}
    \end{minipage}
 \end{minipage}  
     \hspace{1cm}
 \begin{minipage}{0.3\linewidth}
    \begin{minipage}{0.24\linewidth}
        \centering
        \includegraphics[width=\textwidth]{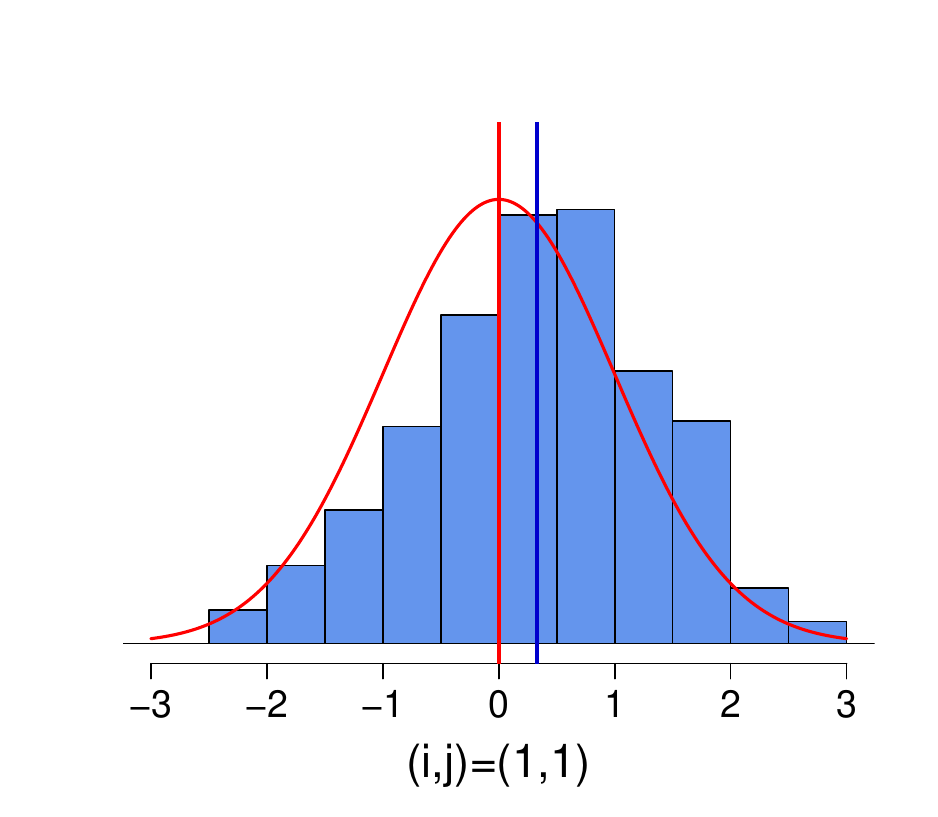}
    \end{minipage}
    \begin{minipage}{0.24\linewidth}
        \centering
        \includegraphics[width=\textwidth]{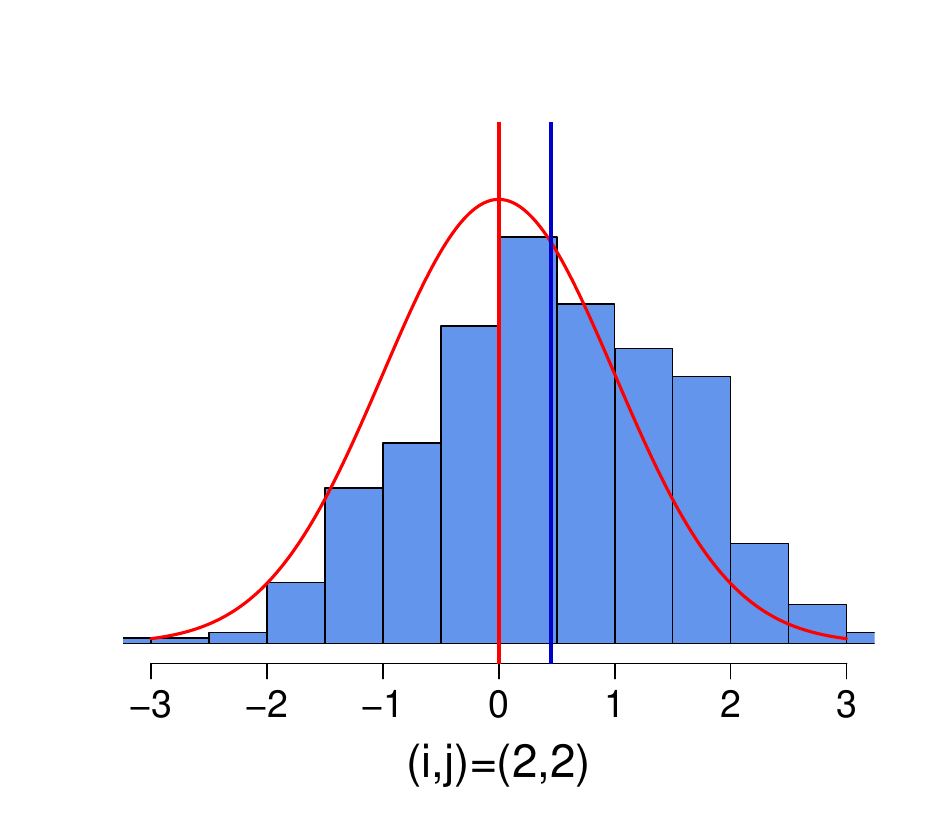}
    \end{minipage}
    \begin{minipage}{0.24\linewidth}
        \centering
        \includegraphics[width=\textwidth]{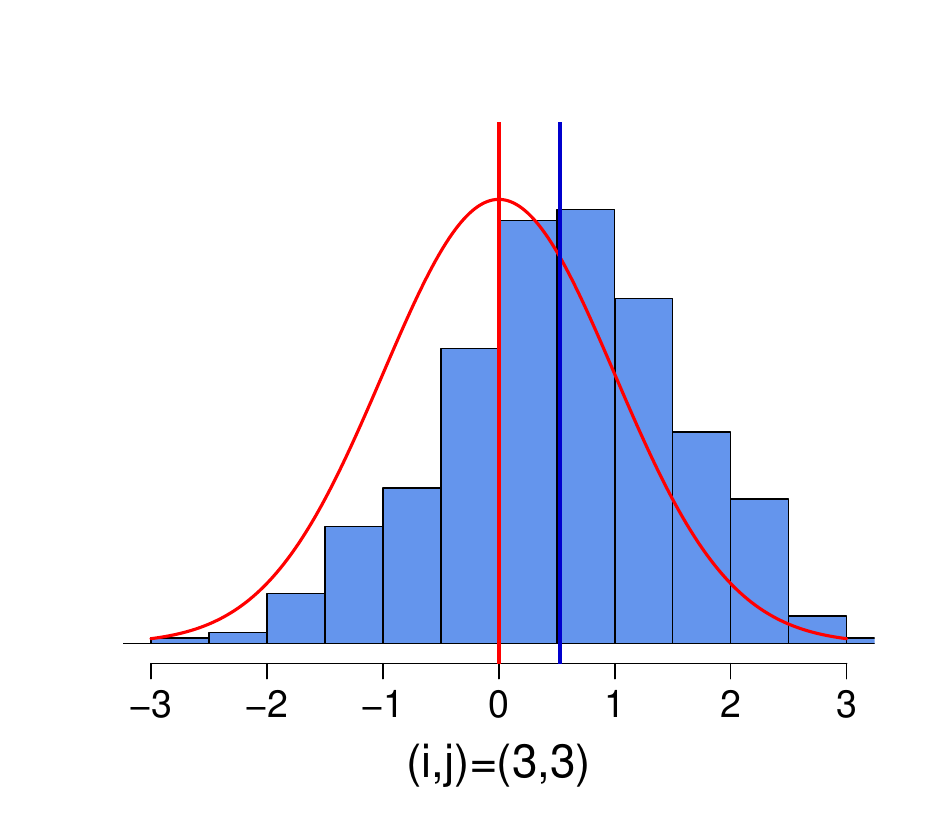}
    \end{minipage}
    \begin{minipage}{0.24\linewidth}
        \centering
        \includegraphics[width=\textwidth]{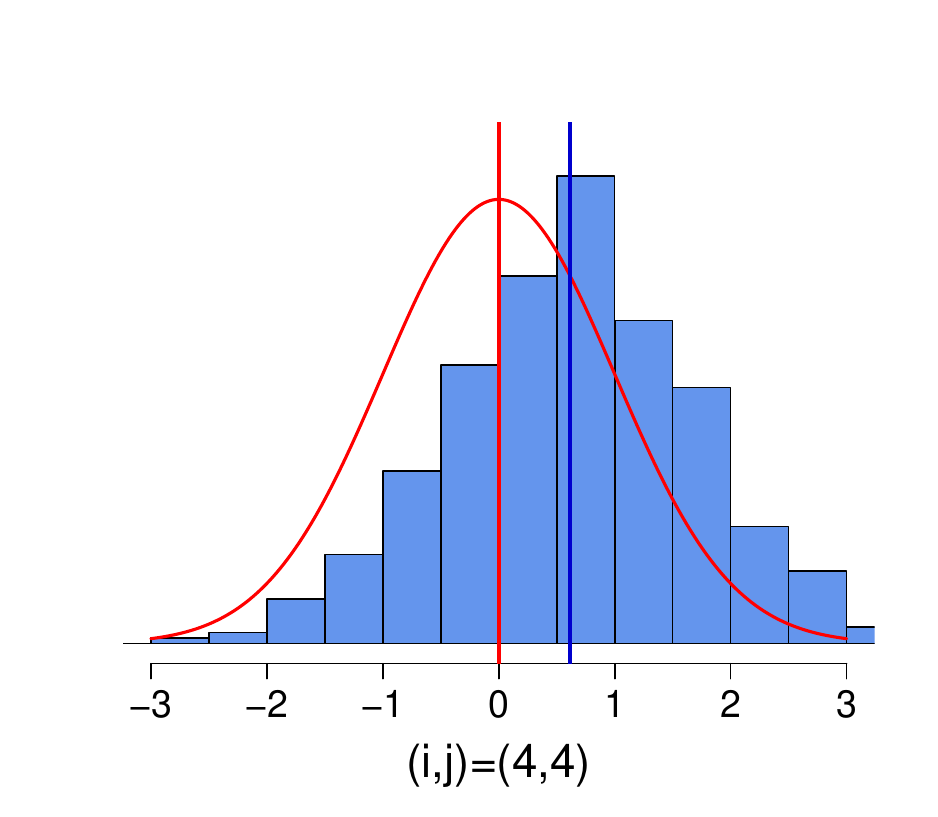}
    \end{minipage}
  \end{minipage}  
    \hspace{1cm}
 \begin{minipage}{0.3\linewidth}
    \begin{minipage}{0.24\linewidth}
        \centering
        \includegraphics[width=\textwidth]{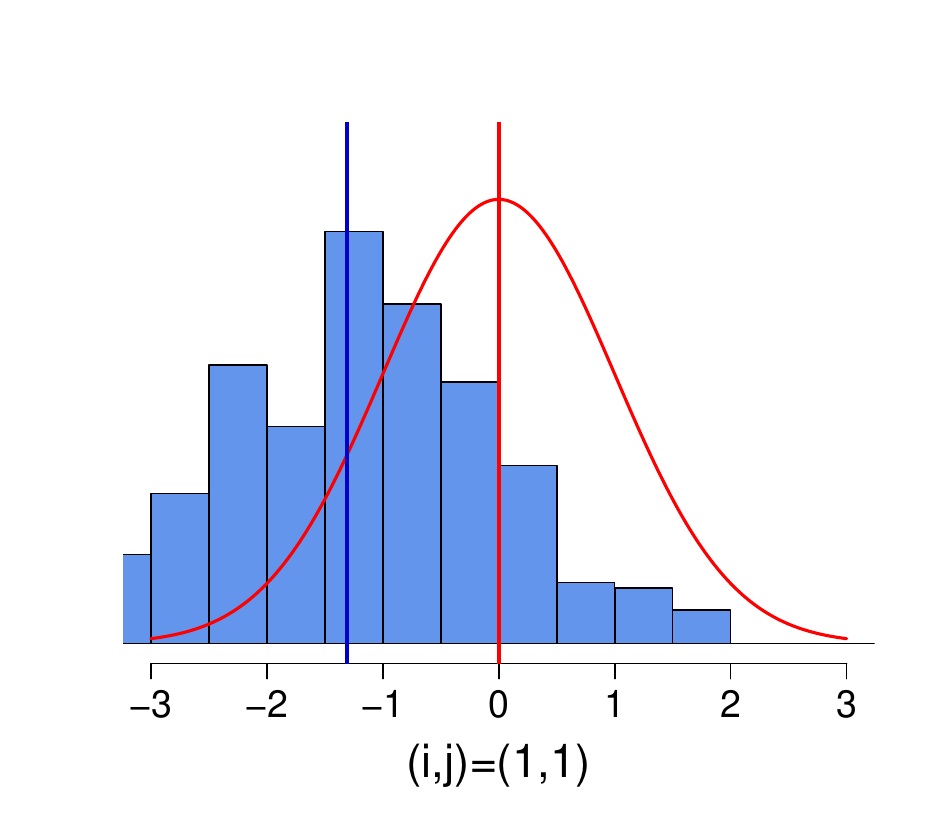}
    \end{minipage}
    \begin{minipage}{0.24\linewidth}
        \centering
        \includegraphics[width=\textwidth]{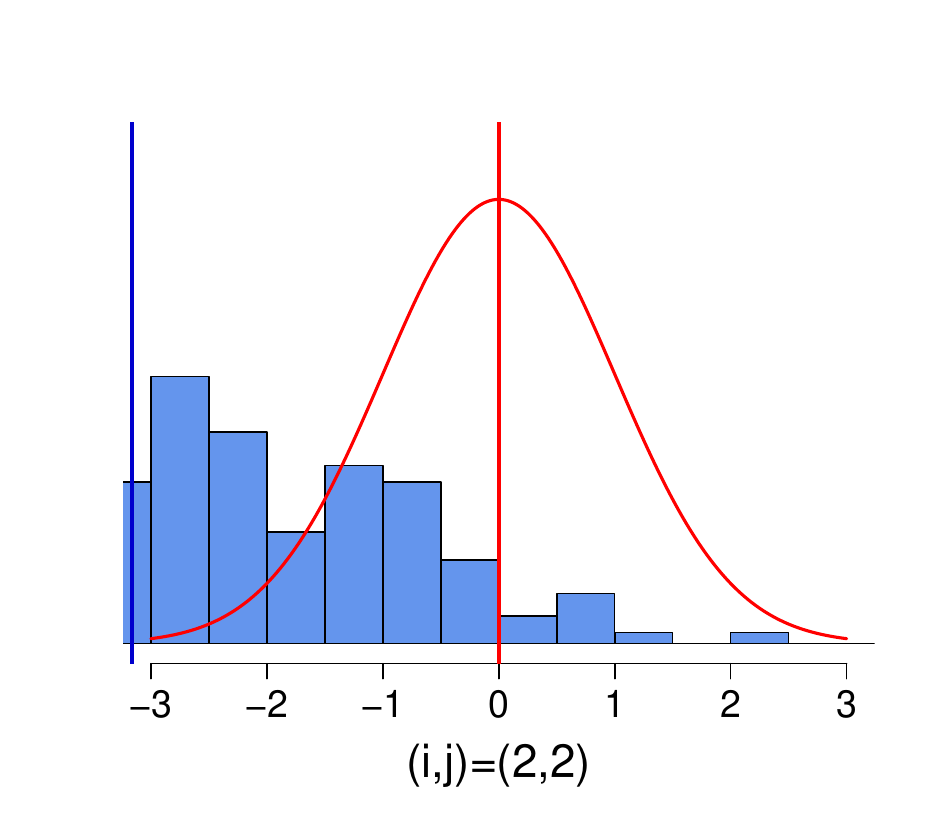}
    \end{minipage}
    \begin{minipage}{0.24\linewidth}
        \centering
        \includegraphics[width=\textwidth]{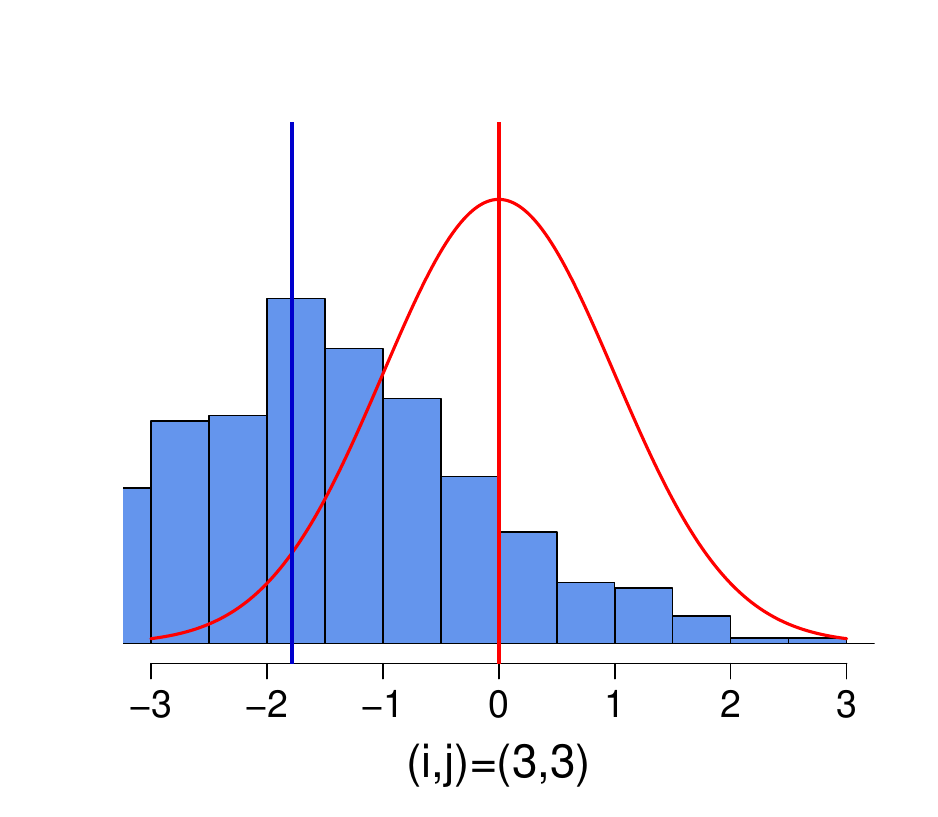}
    \end{minipage}
    \begin{minipage}{0.24\linewidth}
        \centering
        \includegraphics[width=\textwidth]{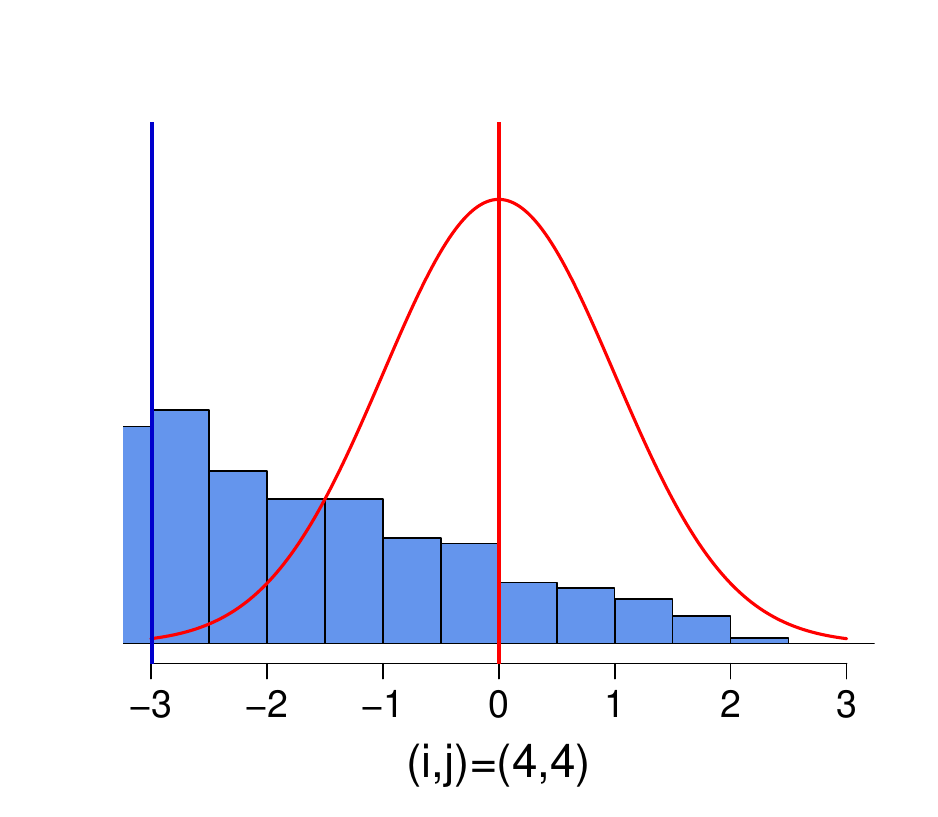}
    \end{minipage}
 \end{minipage}

  \caption*{$n=800, p=400$}
      \vspace{-0.43cm}
 \begin{minipage}{0.3\linewidth}
    \begin{minipage}{0.24\linewidth}
        \centering
        \includegraphics[width=\textwidth]{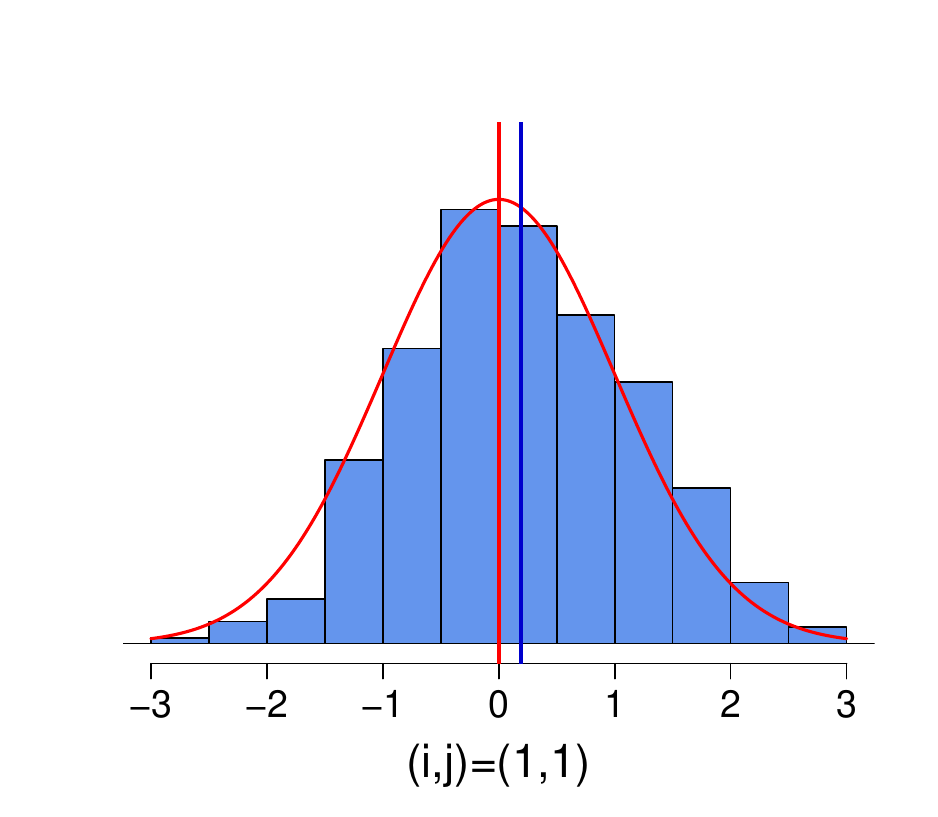}
    \end{minipage}
    \begin{minipage}{0.24\linewidth}
        \centering
        \includegraphics[width=\textwidth]{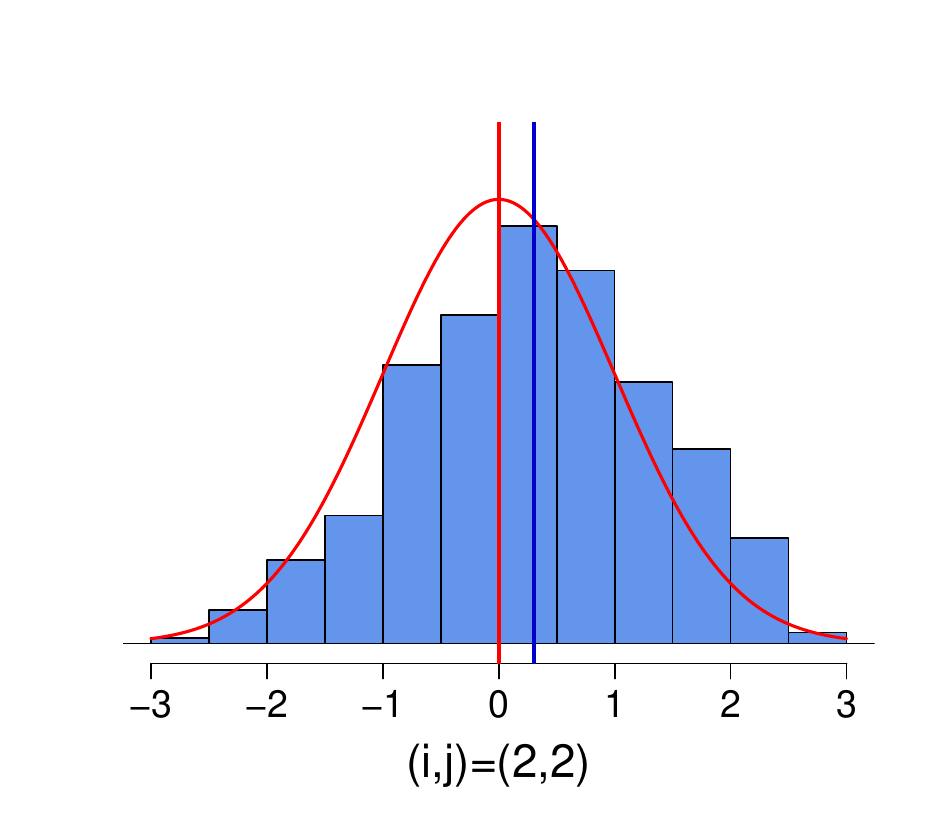}
    \end{minipage}
    \begin{minipage}{0.24\linewidth}
        \centering
        \includegraphics[width=\textwidth]{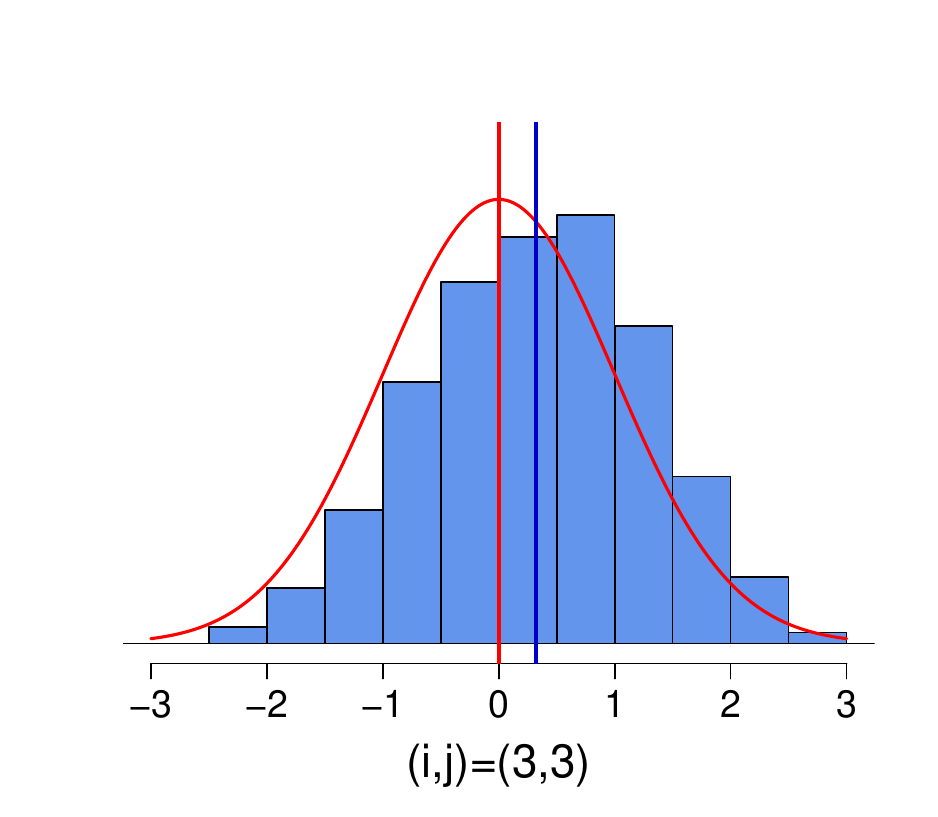}
    \end{minipage}
    \begin{minipage}{0.24\linewidth}
        \centering
        \includegraphics[width=\textwidth]{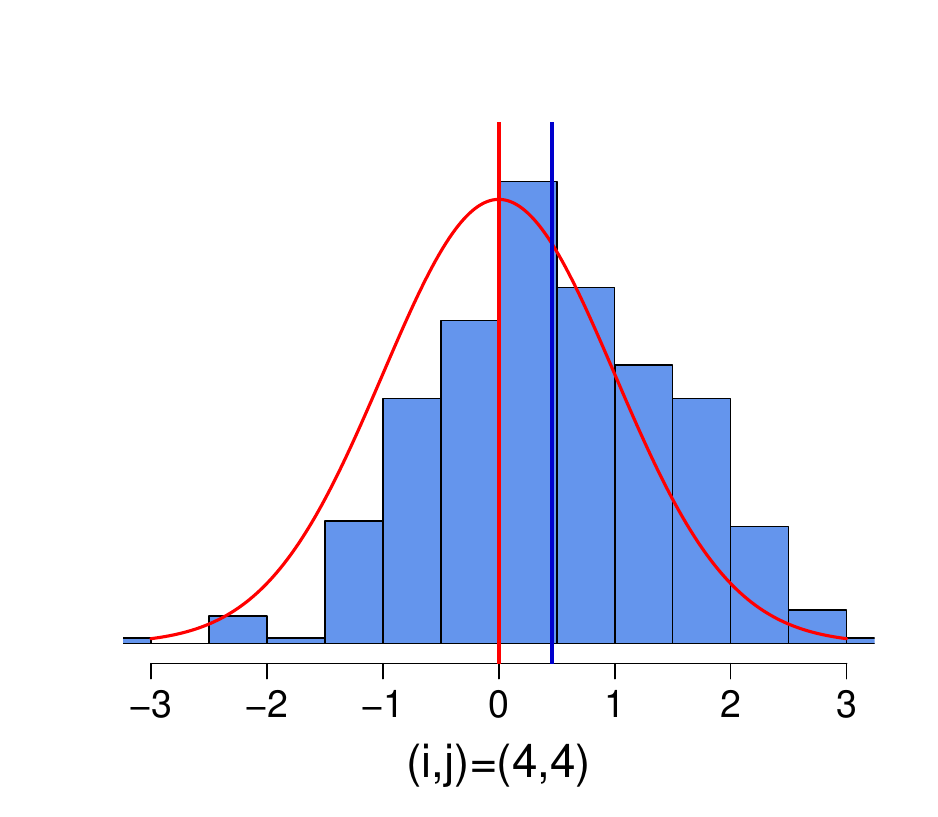}
    \end{minipage}
   \caption*{(a)~~$L_0{:}~ \widehat{\mb{\Omega}}^{\text{US}}$}
 \end{minipage} 
     \hspace{1cm}
 \begin{minipage}{0.3\linewidth}
    \begin{minipage}{0.24\linewidth}
        \centering
        \includegraphics[width=\textwidth]{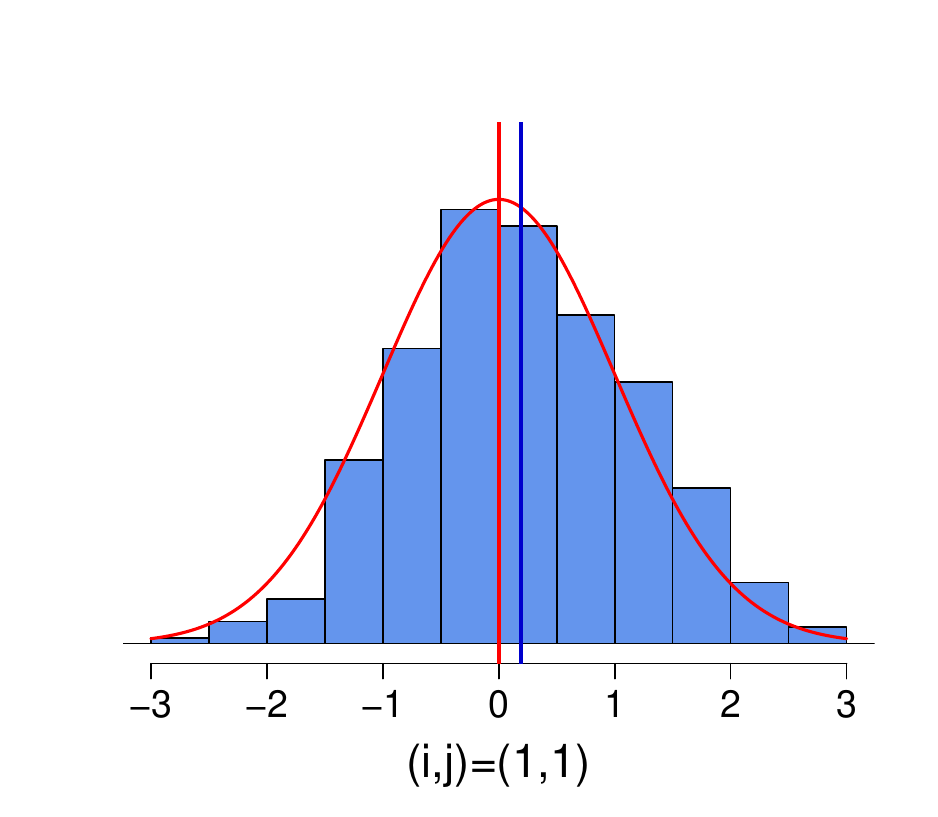}
    \end{minipage}
    \begin{minipage}{0.24\linewidth}
        \centering
        \includegraphics[width=\textwidth]{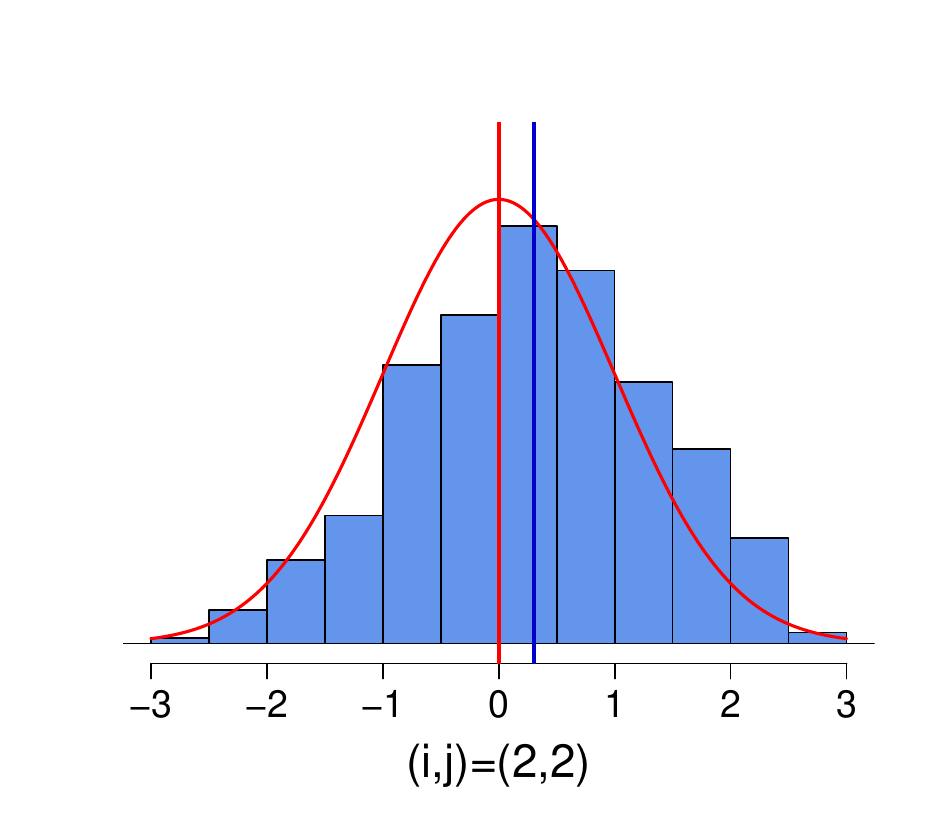}
    \end{minipage}
    \begin{minipage}{0.24\linewidth}
        \centering
        \includegraphics[width=\textwidth]{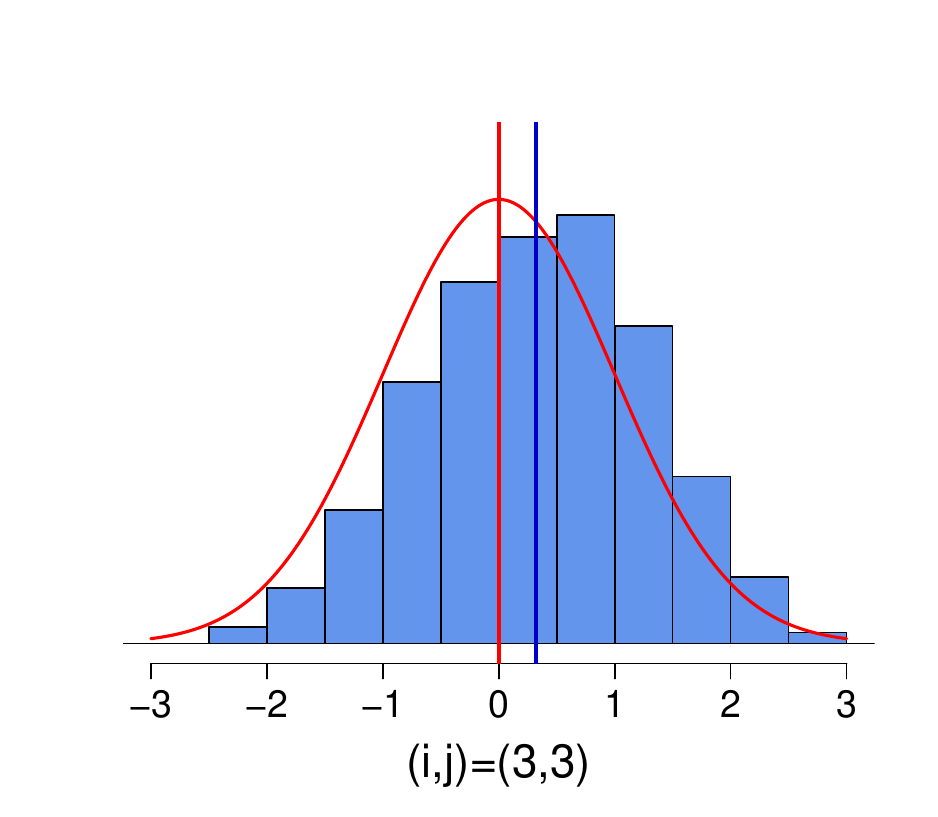}
    \end{minipage}
    \begin{minipage}{0.24\linewidth}
        \centering
        \includegraphics[width=\textwidth]{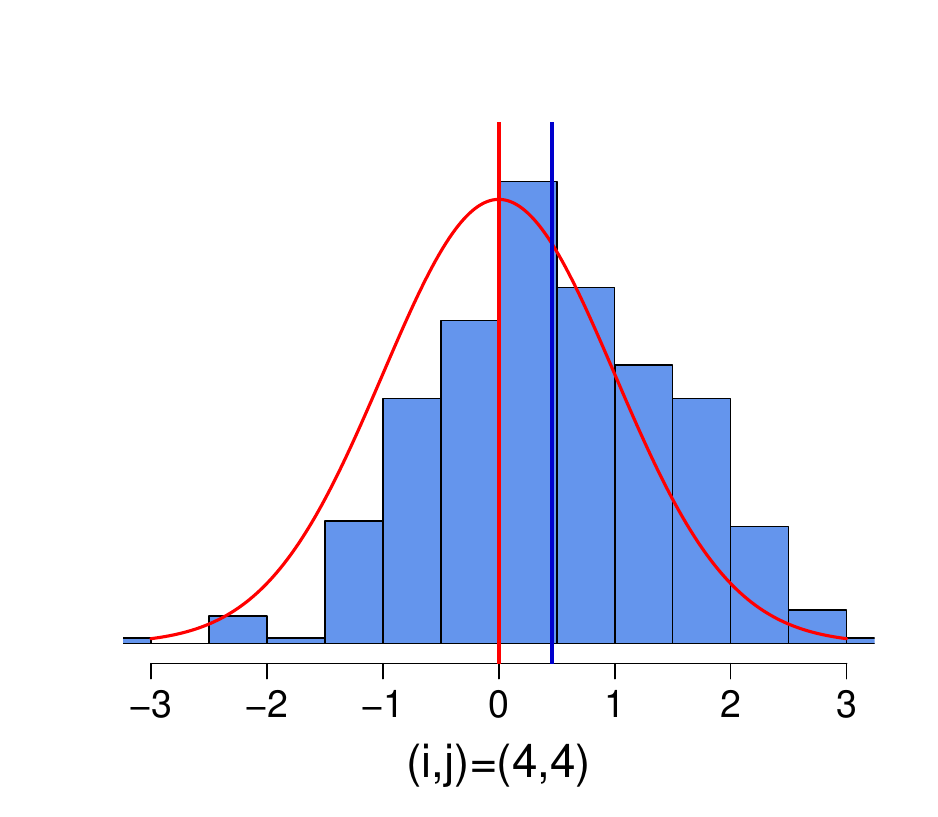}
    \end{minipage}
        \caption*{(b)~~$L_0{:}~ \widehat{\mb{T}}$}
 \end{minipage}   
      \hspace{1cm}
 \begin{minipage}{0.3\linewidth}
    \begin{minipage}{0.24\linewidth}
        \centering
        \includegraphics[width=\textwidth]{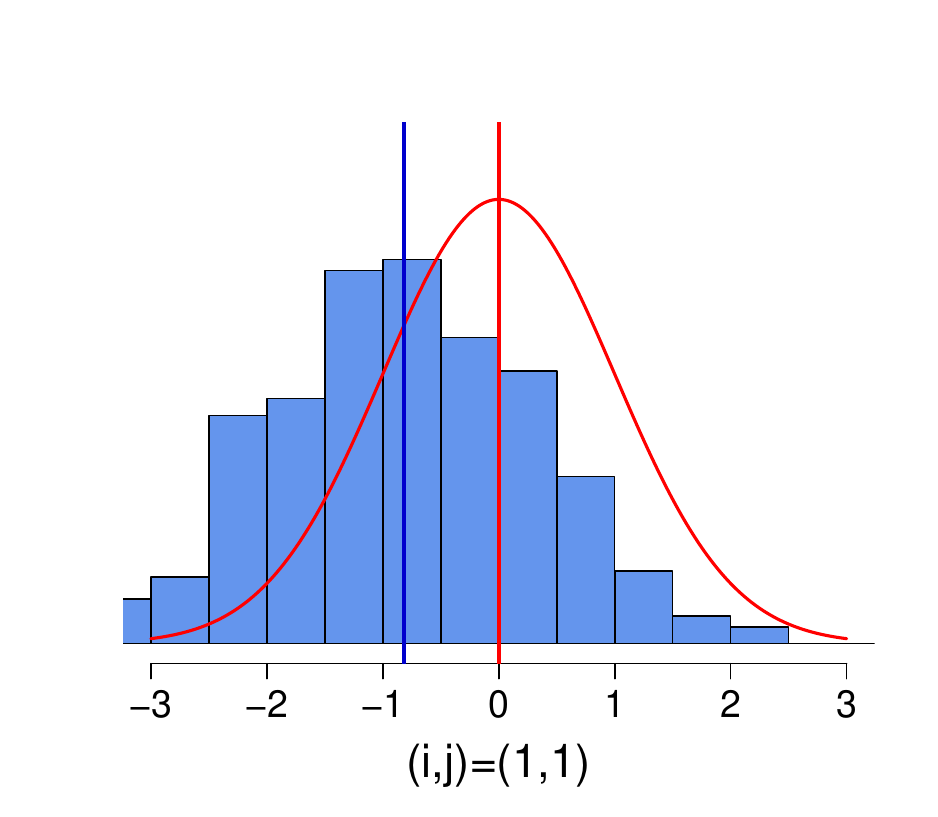}
    \end{minipage}
    \begin{minipage}{0.24\linewidth}
        \centering
        \includegraphics[width=\textwidth]{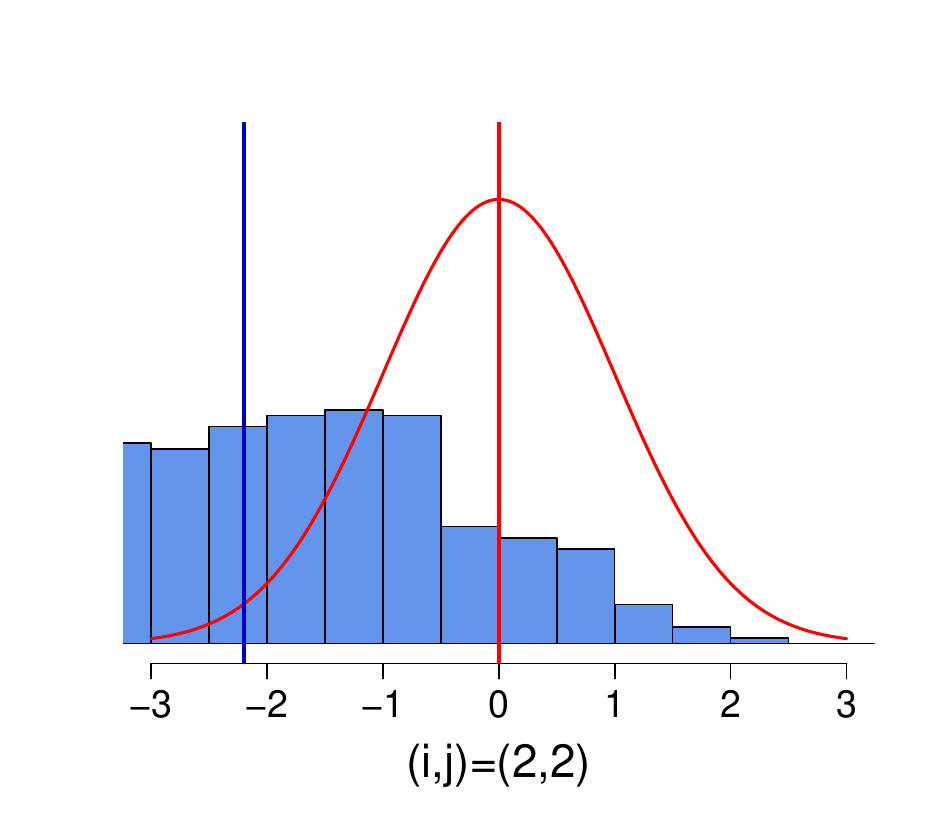}
    \end{minipage}
    \begin{minipage}{0.24\linewidth}
        \centering
        \includegraphics[width=\textwidth]{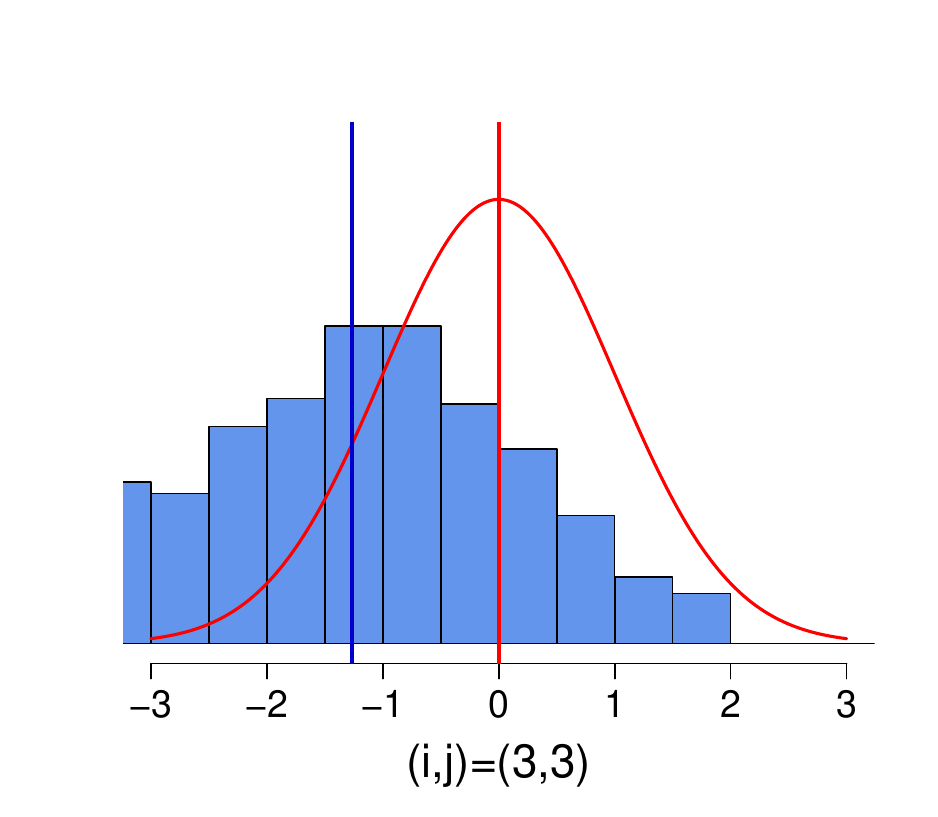}
    \end{minipage}
    \begin{minipage}{0.24\linewidth}
        \centering
        \includegraphics[width=\textwidth]{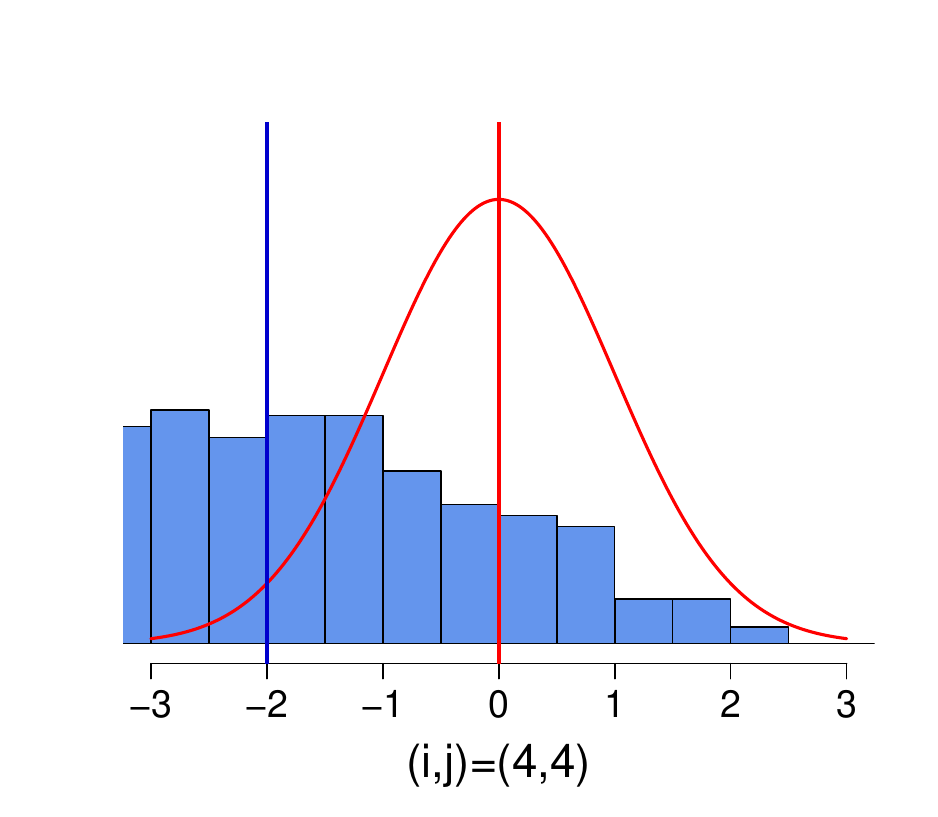}
    \end{minipage}
    \caption*{(c)~~$L_1{:}~ \widehat{\mb{T}}$}
     \end{minipage}   
     \caption{Histograms of $\big(\sqrt{n}(\widehat{\mb{\Omega}}_{ij}^{(m)}-\mb{\Omega}_{ij})/\widehat{\sigma}_{\mb{\Omega}_{ij}}^{(m)}\big)_{m=1}^{400}$ under sub-Gaussian random graph settings.}
\end{sidewaysfigure}

 %sub-Gaussian hub
 \begin{sidewaysfigure}[th!]
  \caption*{$n=200, p=200$}
      \vspace{-0.43cm}
 \begin{minipage}{0.3\linewidth}
    \begin{minipage}{0.24\linewidth}
        \centering
        \includegraphics[width=\textwidth]{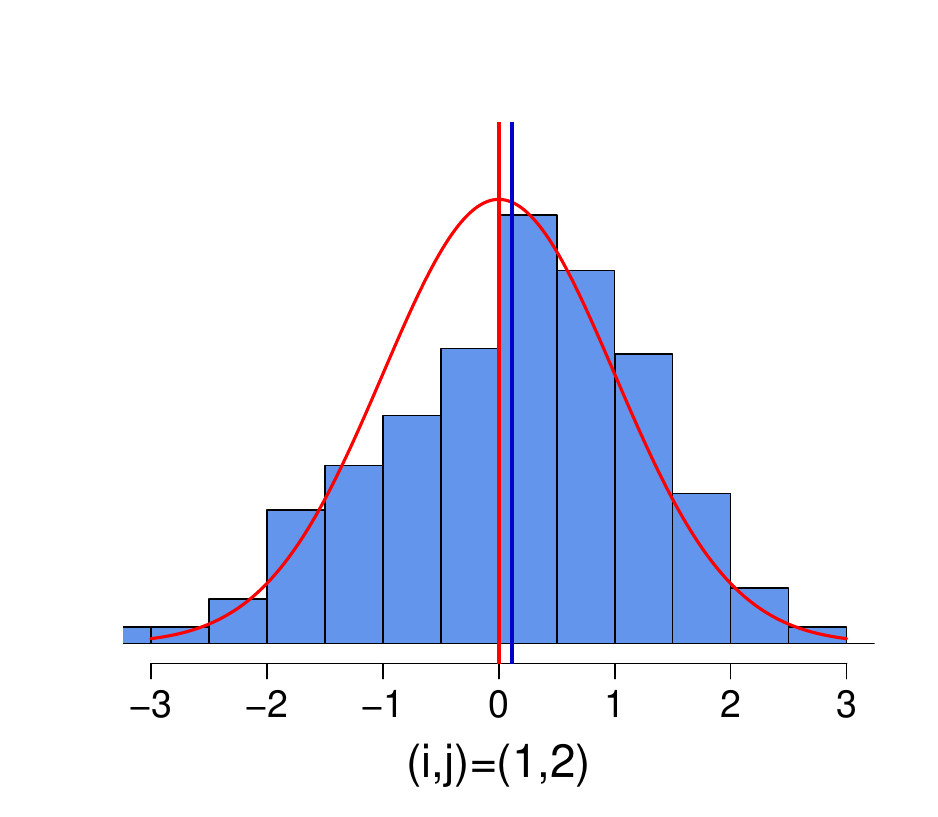}
    \end{minipage}
    \begin{minipage}{0.24\linewidth}
        \centering
        \includegraphics[width=\textwidth]{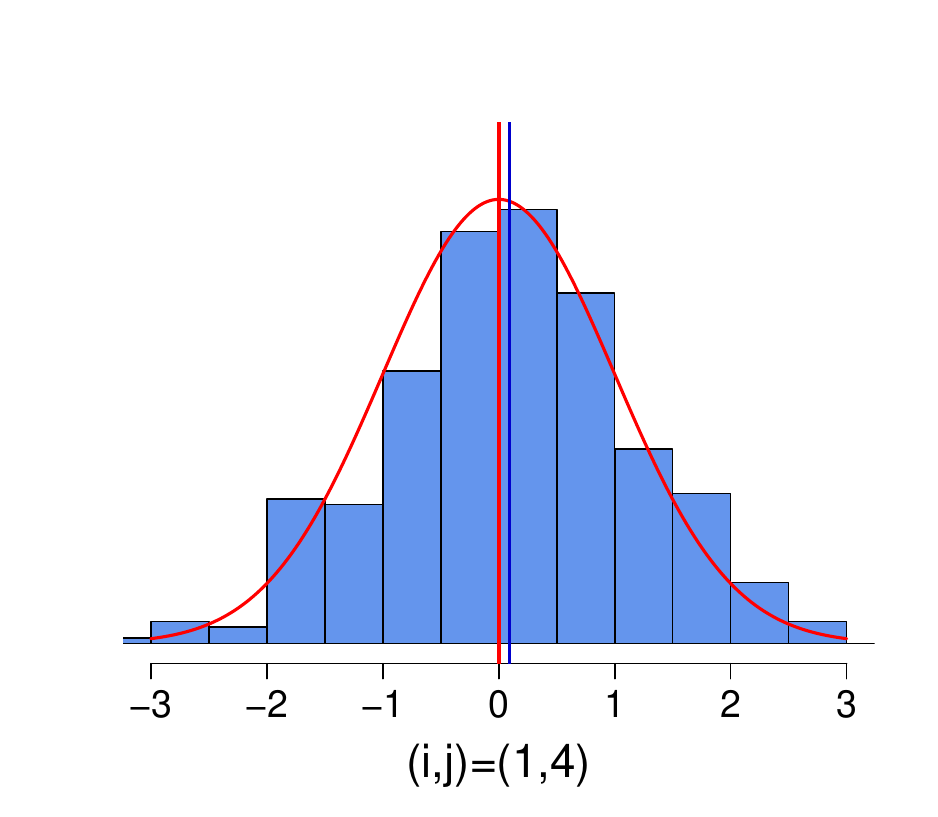}
    \end{minipage}
    \begin{minipage}{0.24\linewidth}
        \centering
        \includegraphics[width=\textwidth]{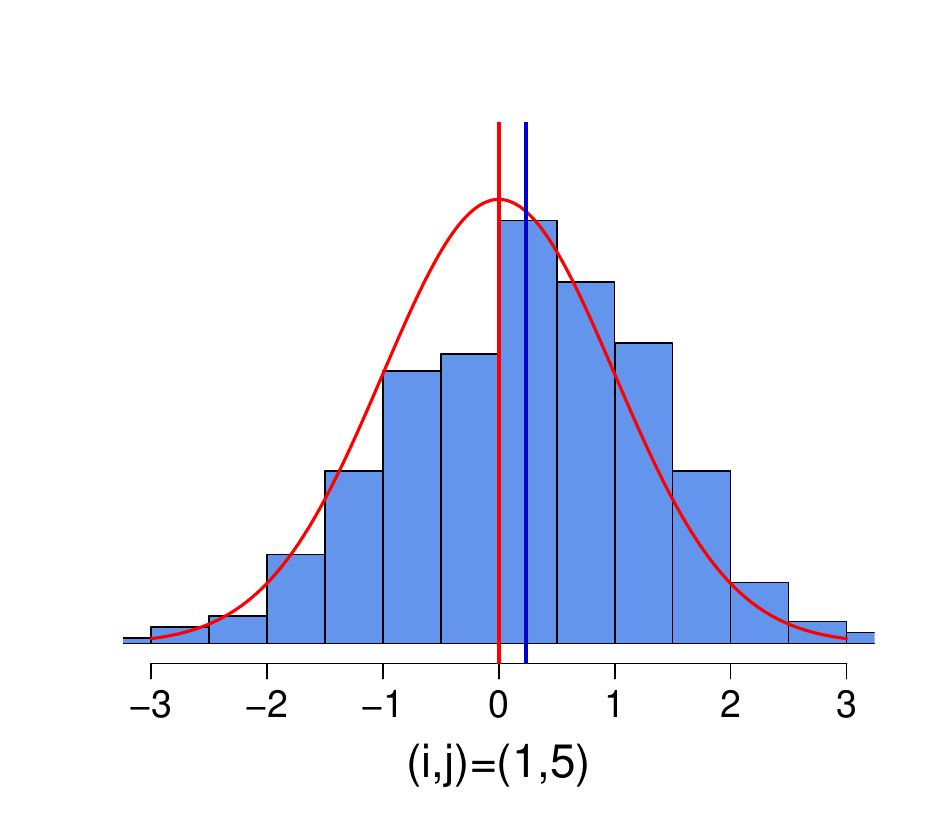}
    \end{minipage}
    \begin{minipage}{0.24\linewidth}
        \centering
        \includegraphics[width=\textwidth]{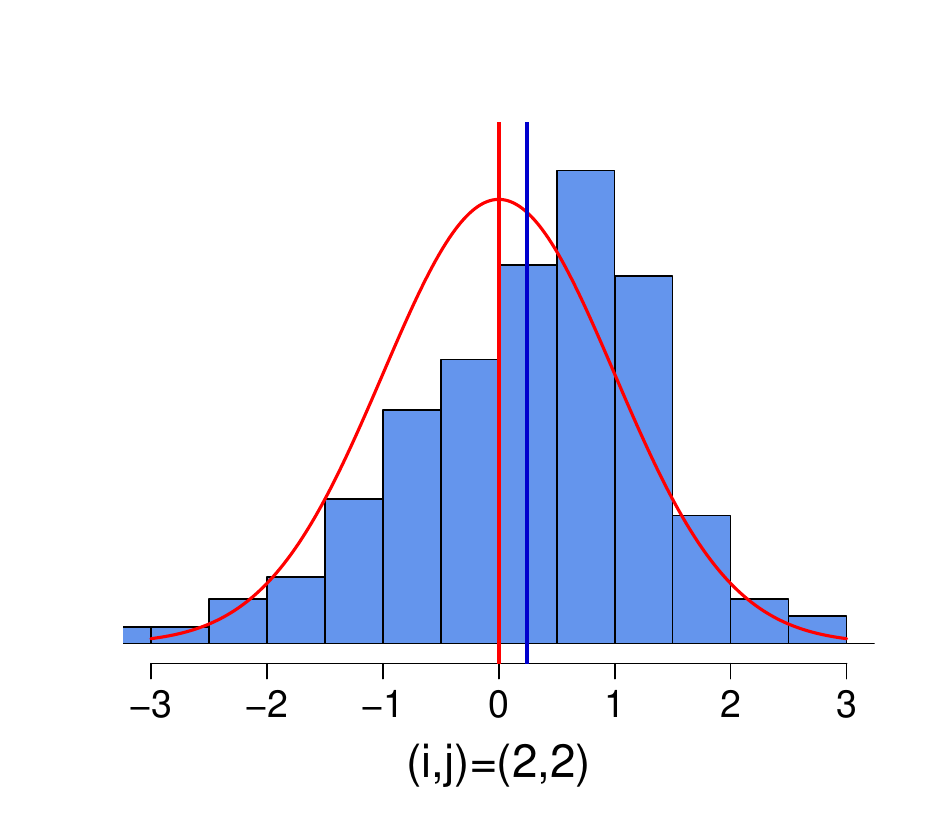}
    \end{minipage}
 \end{minipage}
 \hspace{1cm}
 \begin{minipage}{0.3\linewidth}
    \begin{minipage}{0.24\linewidth}
        \centering
        \includegraphics[width=\textwidth]{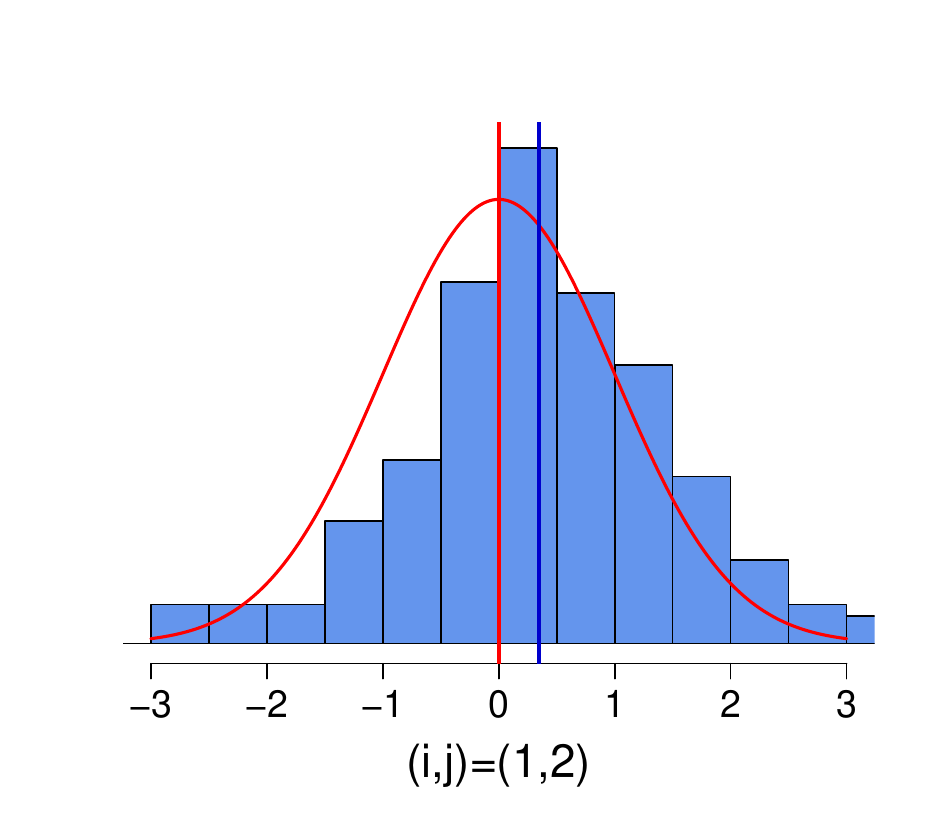}
    \end{minipage}
    \begin{minipage}{0.24\linewidth}
        \centering
        \includegraphics[width=\textwidth]{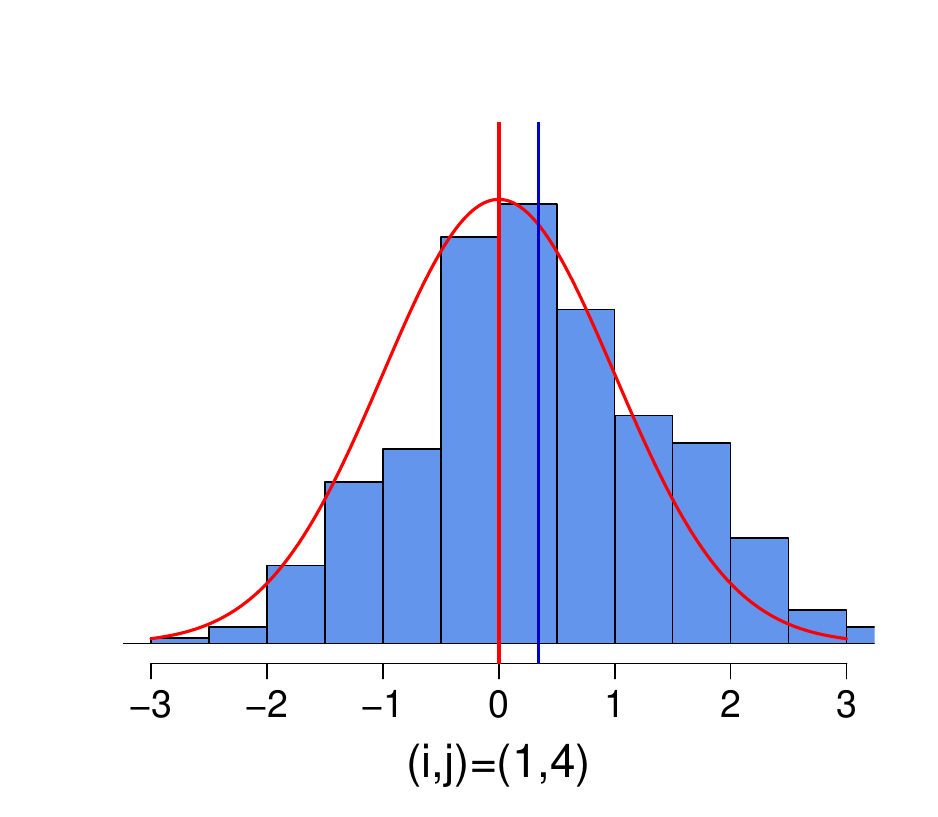}
    \end{minipage}
    \begin{minipage}{0.24\linewidth}
        \centering
        \includegraphics[width=\textwidth]{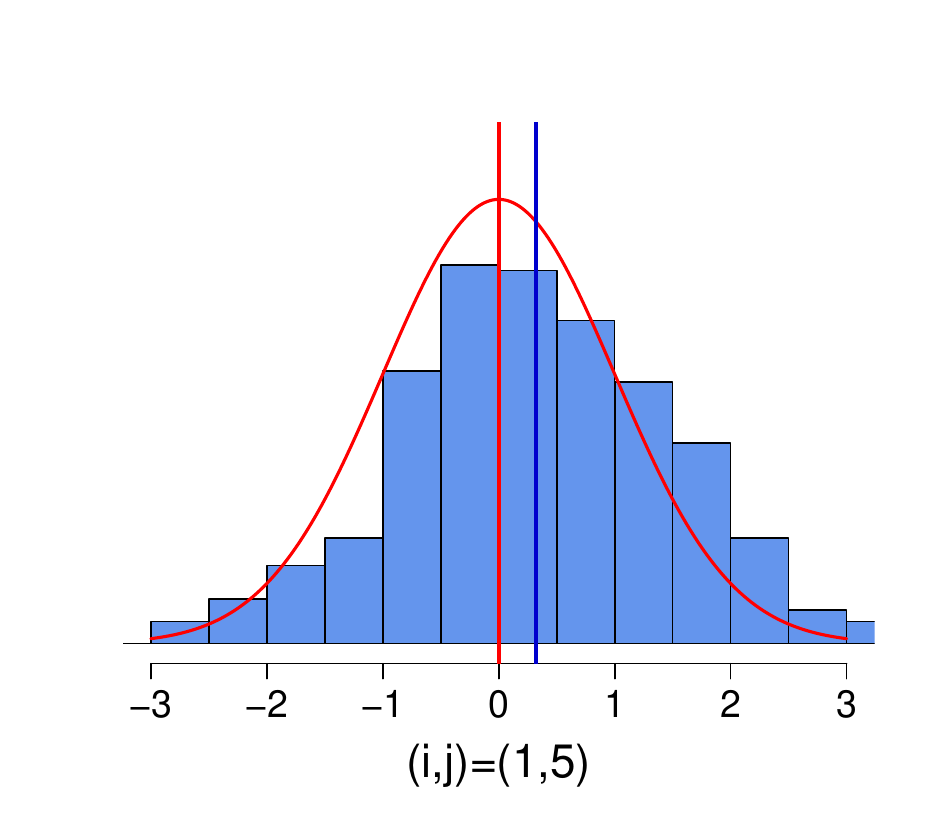}
    \end{minipage}
    \begin{minipage}{0.24\linewidth}
        \centering
        \includegraphics[width=\textwidth]{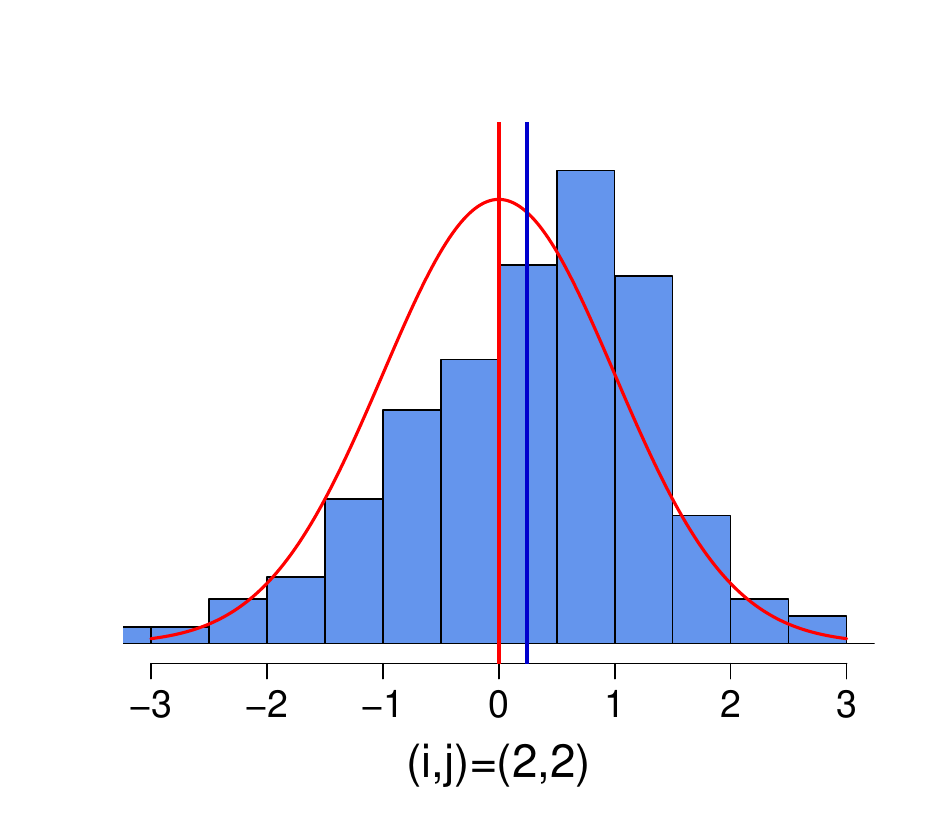}
    \end{minipage}    
 \end{minipage}
  \hspace{1cm}
 \begin{minipage}{0.3\linewidth}
     \begin{minipage}{0.24\linewidth}
        \centering
        \includegraphics[width=\textwidth]{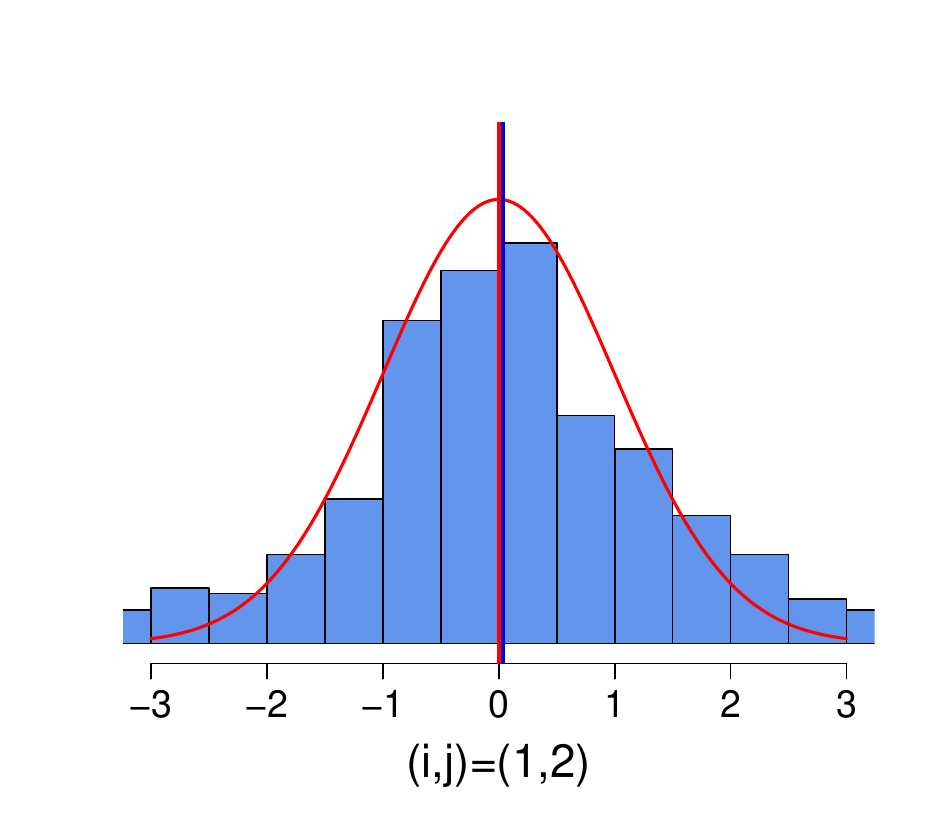}
    \end{minipage}
    \begin{minipage}{0.24\linewidth}
        \centering
        \includegraphics[width=\textwidth]{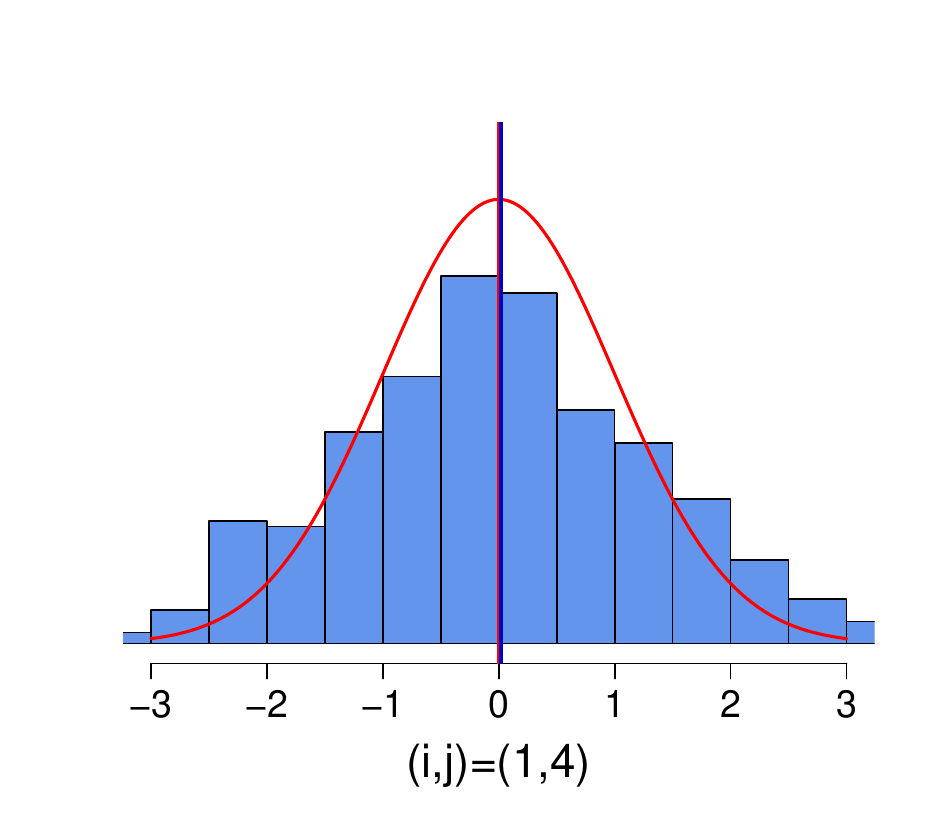}
    \end{minipage}
    \begin{minipage}{0.24\linewidth}
        \centering
        \includegraphics[width=\textwidth]{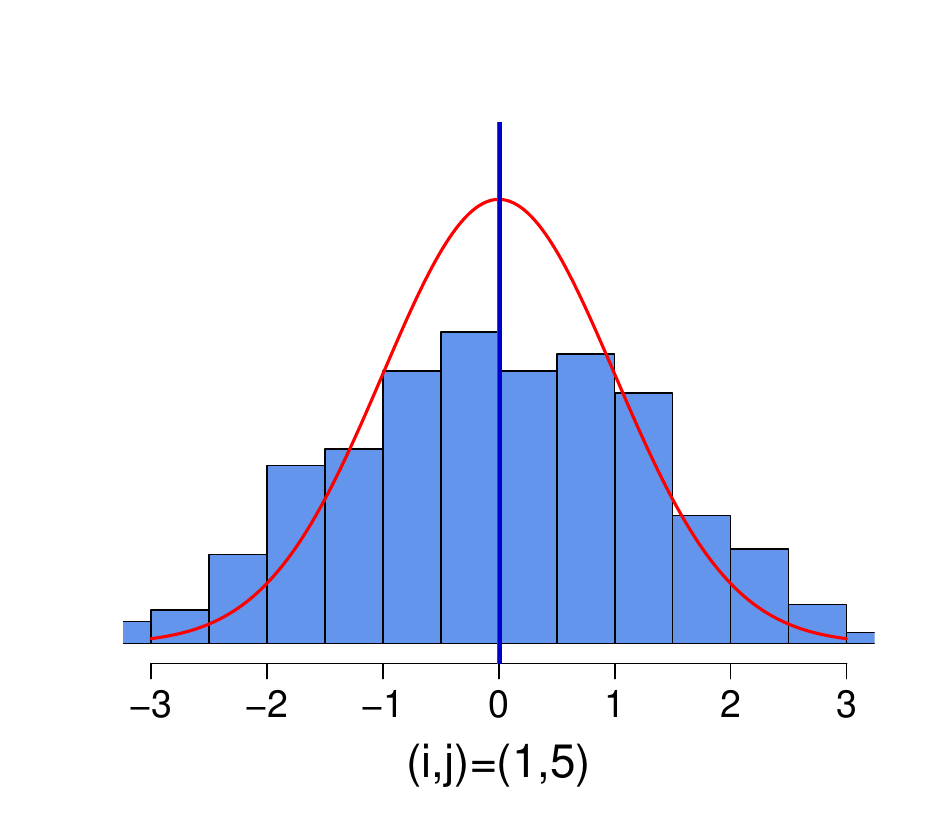}
    \end{minipage}
    \begin{minipage}{0.24\linewidth}
        \centering
        \includegraphics[width=\textwidth]{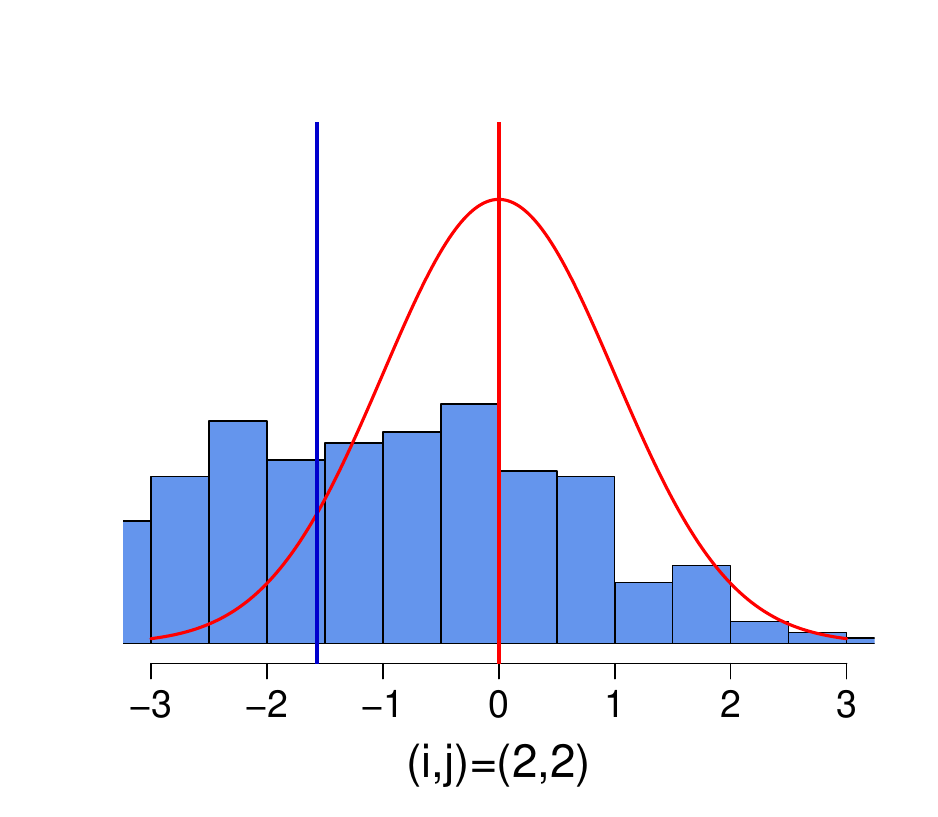}
    \end{minipage}
 \end{minipage}

  \caption*{$n=400, p=200$}
      \vspace{-0.43cm}
 \begin{minipage}{0.3\linewidth}
    \begin{minipage}{0.24\linewidth}
        \centering
        \includegraphics[width=\textwidth]{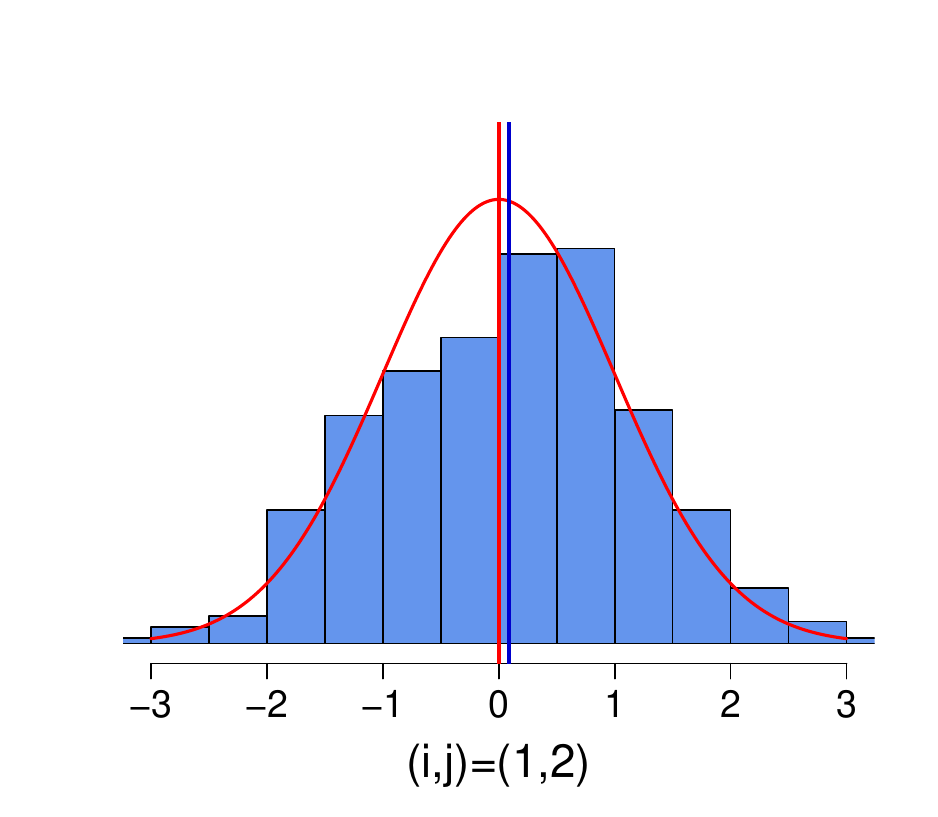}
    \end{minipage}
    \begin{minipage}{0.24\linewidth}
        \centering
        \includegraphics[width=\textwidth]{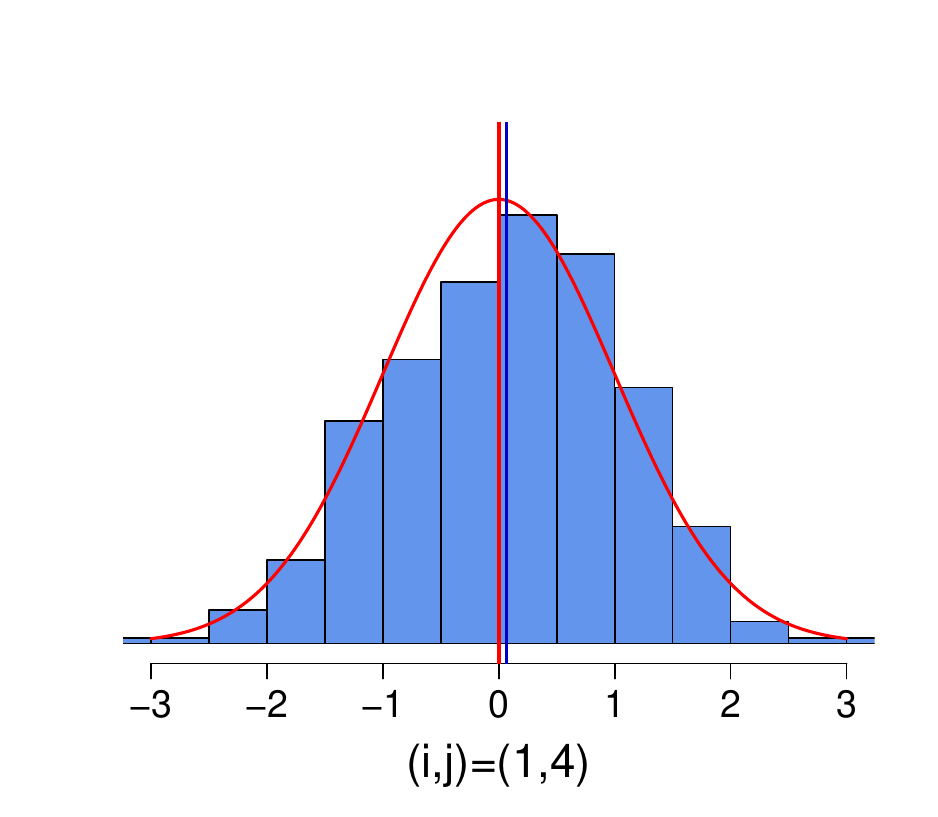}
    \end{minipage}
    \begin{minipage}{0.24\linewidth}
        \centering
        \includegraphics[width=\textwidth]{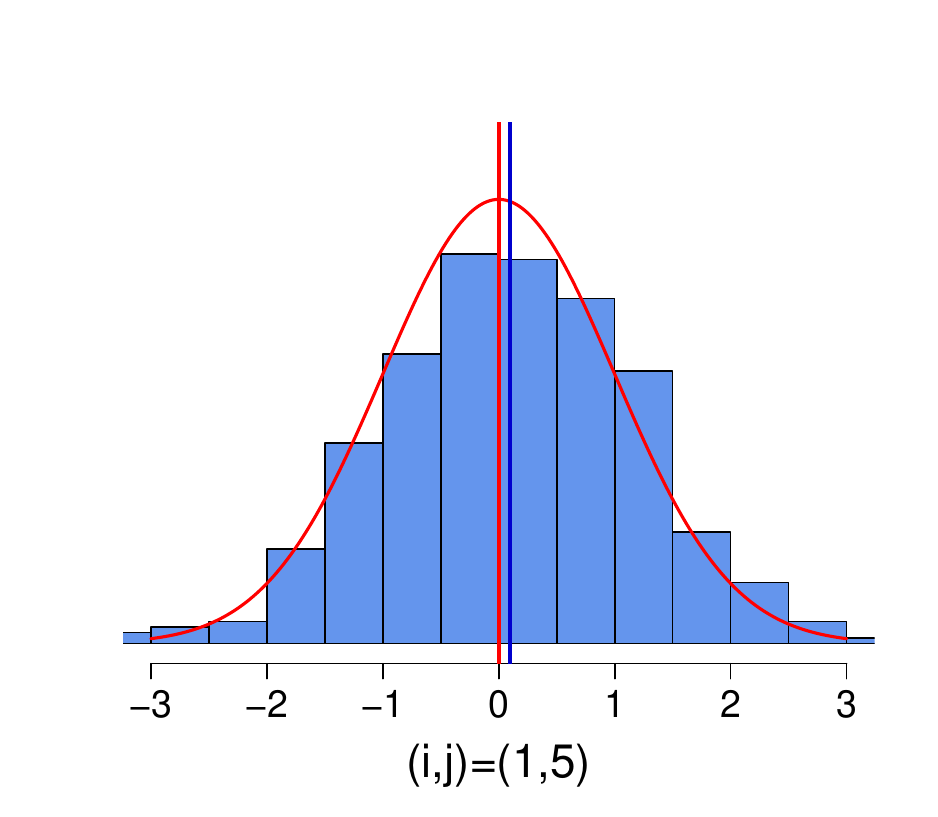}
    \end{minipage}
    \begin{minipage}{0.24\linewidth}
        \centering
        \includegraphics[width=\textwidth]{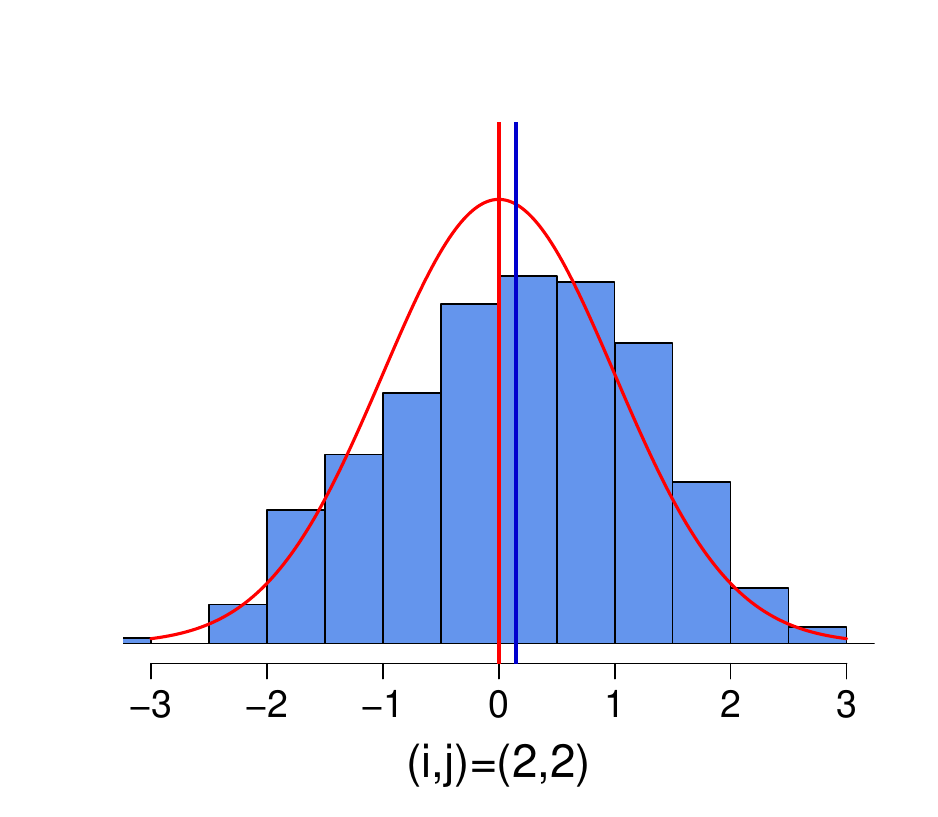}
    \end{minipage}
 \end{minipage}  
     \hspace{1cm}
 \begin{minipage}{0.3\linewidth}
    \begin{minipage}{0.24\linewidth}
        \centering
        \includegraphics[width=\textwidth]{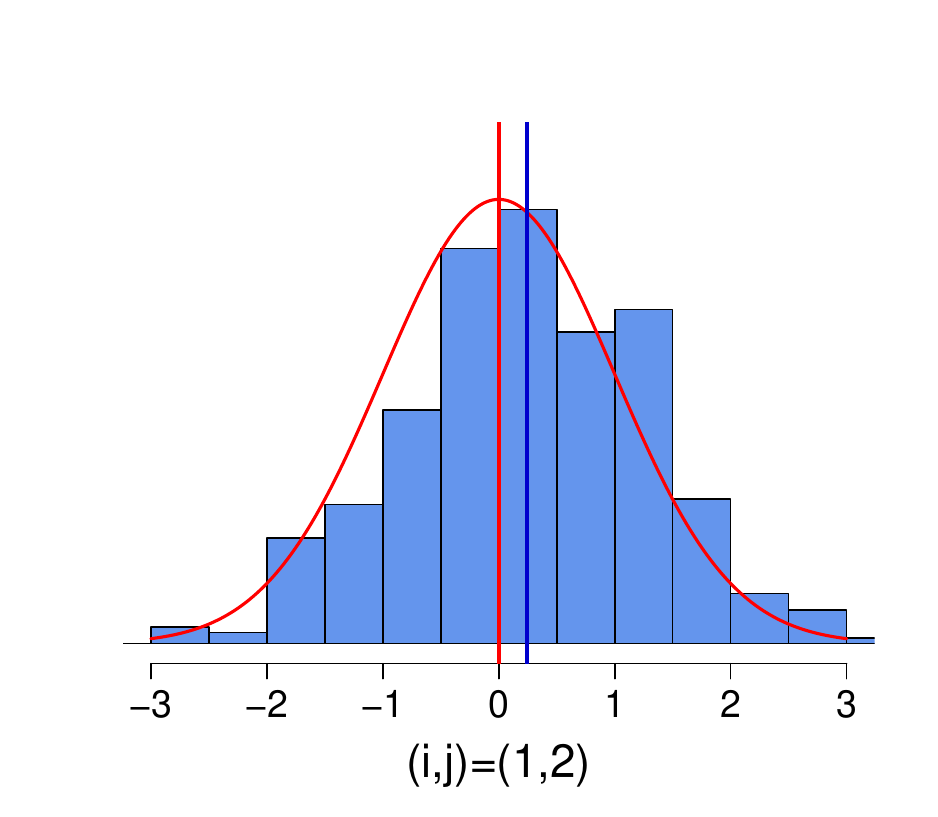}
    \end{minipage}
    \begin{minipage}{0.24\linewidth}
        \centering
        \includegraphics[width=\textwidth]{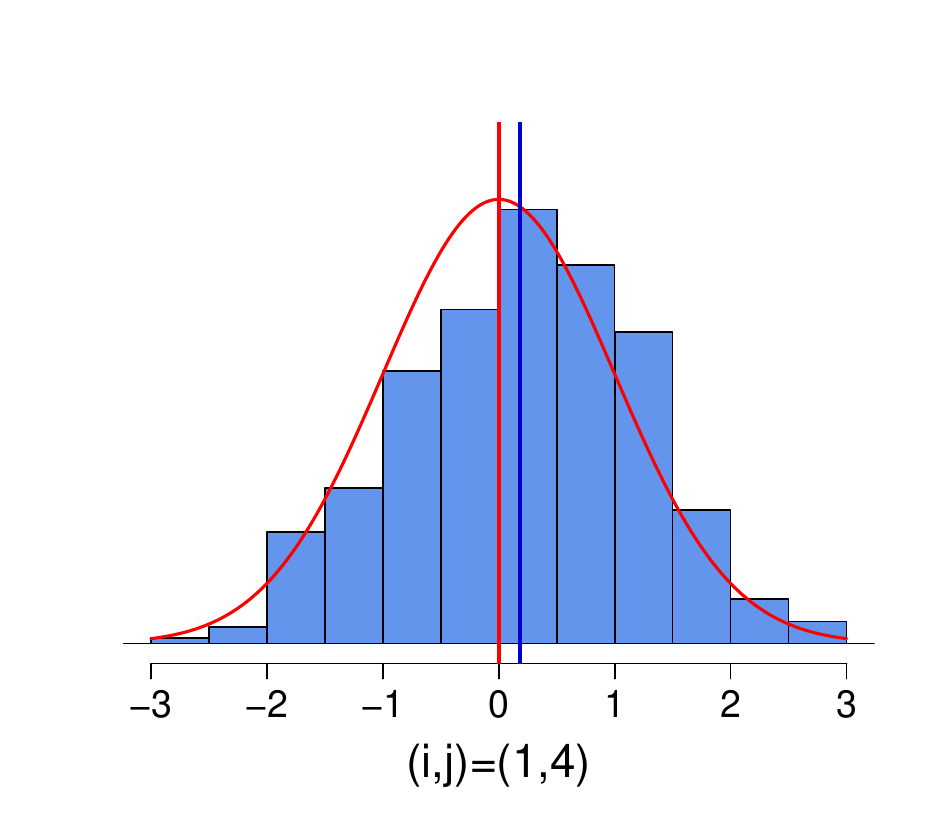}
    \end{minipage}
    \begin{minipage}{0.24\linewidth}
        \centering
        \includegraphics[width=\textwidth]{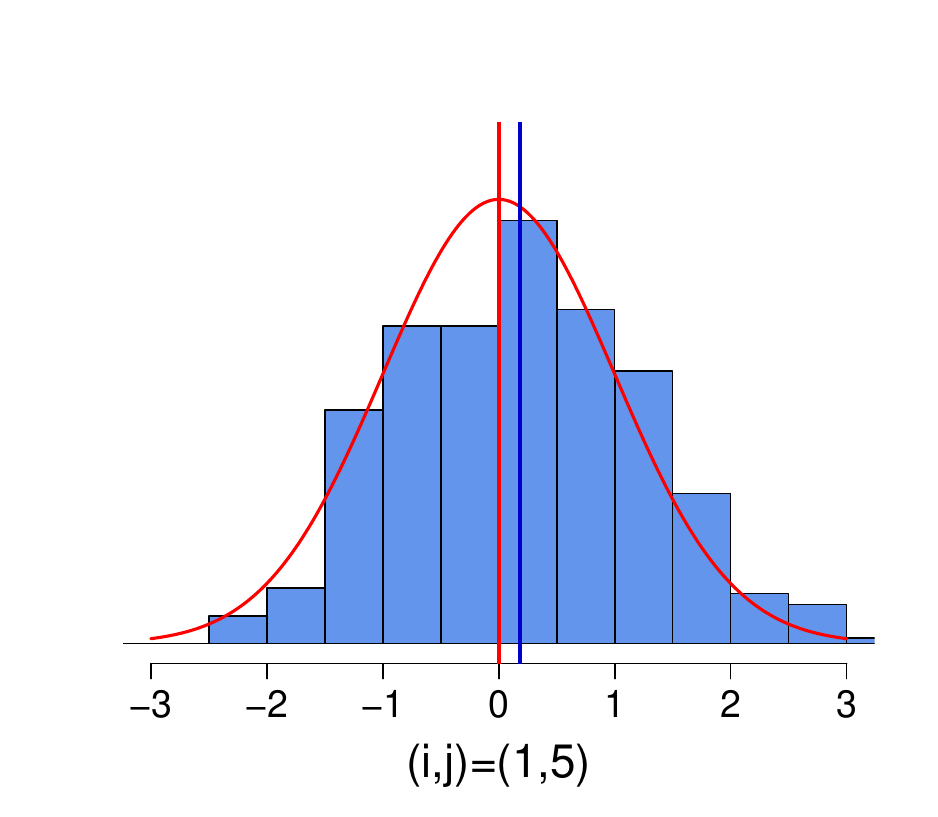}
    \end{minipage}
    \begin{minipage}{0.24\linewidth}
        \centering
        \includegraphics[width=\textwidth]{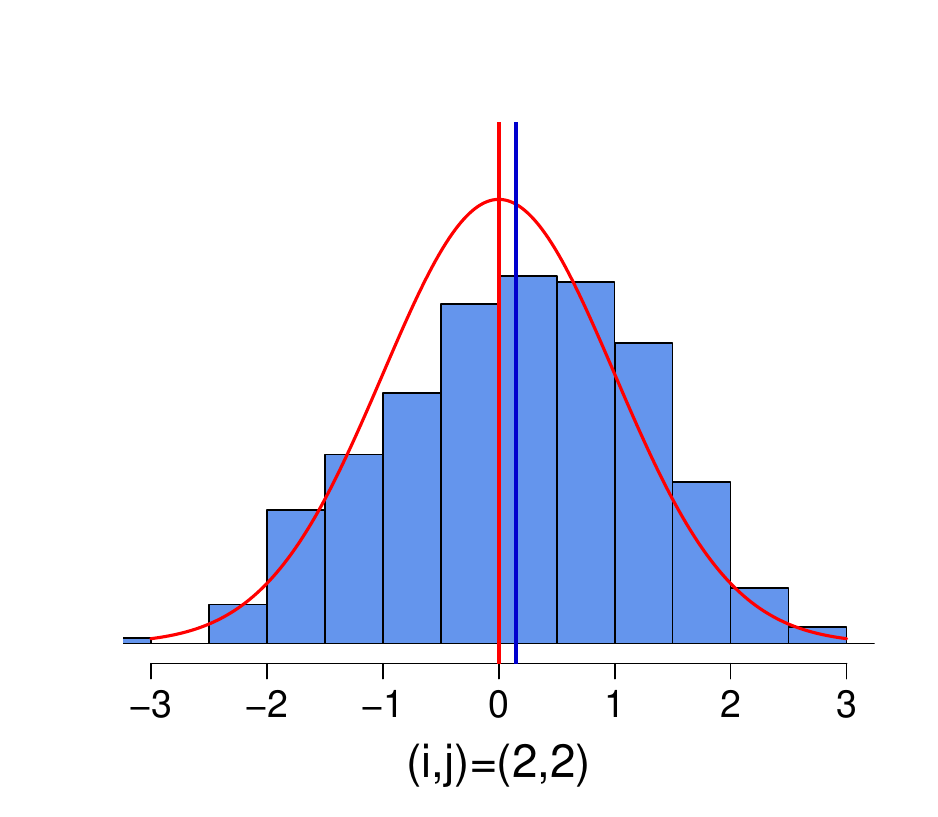}
    \end{minipage}
  \end{minipage}  
    \hspace{1cm}
 \begin{minipage}{0.3\linewidth}
    \begin{minipage}{0.24\linewidth}
        \centering
        \includegraphics[width=\textwidth]{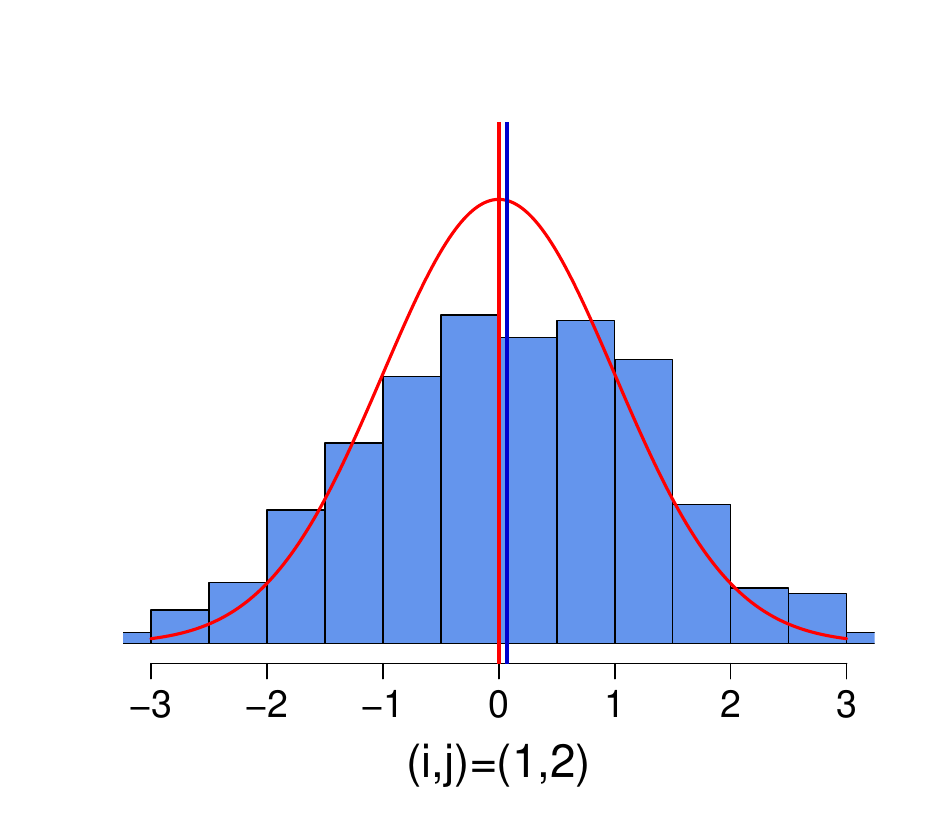}
    \end{minipage}
    \begin{minipage}{0.24\linewidth}
        \centering
        \includegraphics[width=\textwidth]{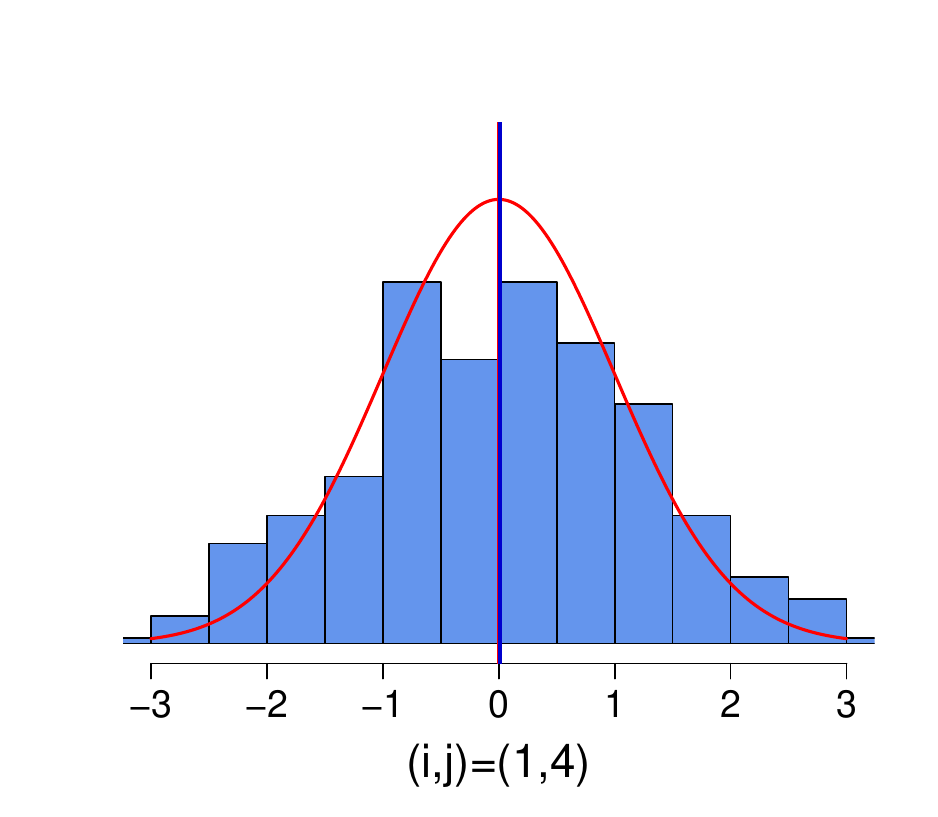}
    \end{minipage}
    \begin{minipage}{0.24\linewidth}
        \centering
        \includegraphics[width=\textwidth]{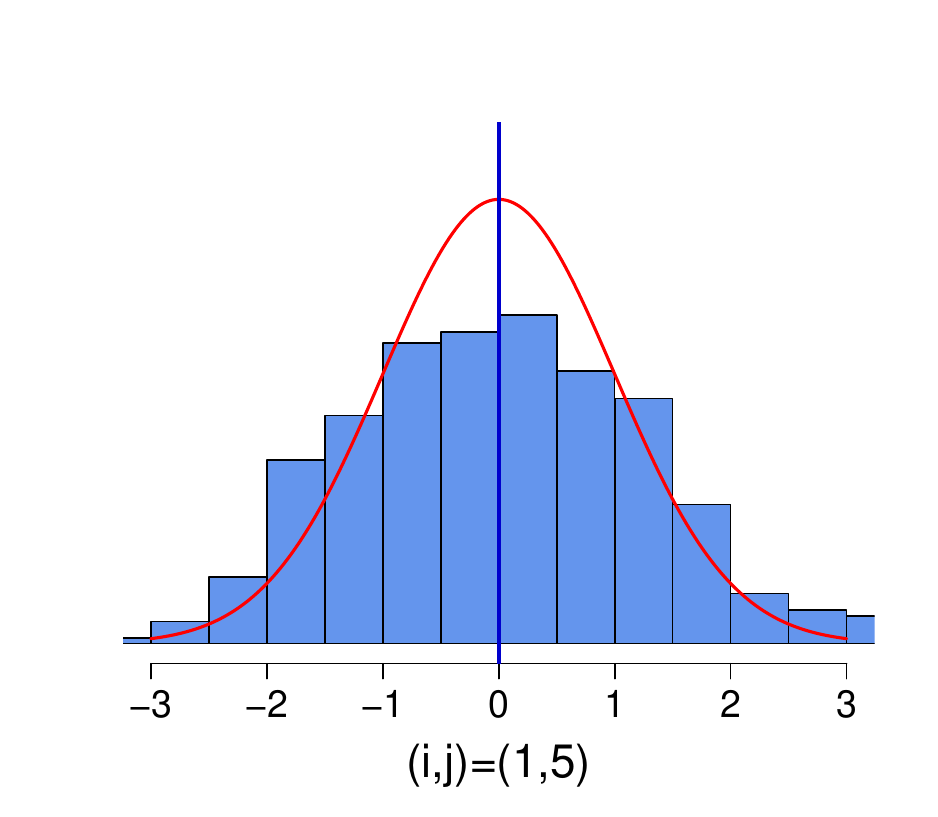}
    \end{minipage}
    \begin{minipage}{0.24\linewidth}
        \centering
        \includegraphics[width=\textwidth]{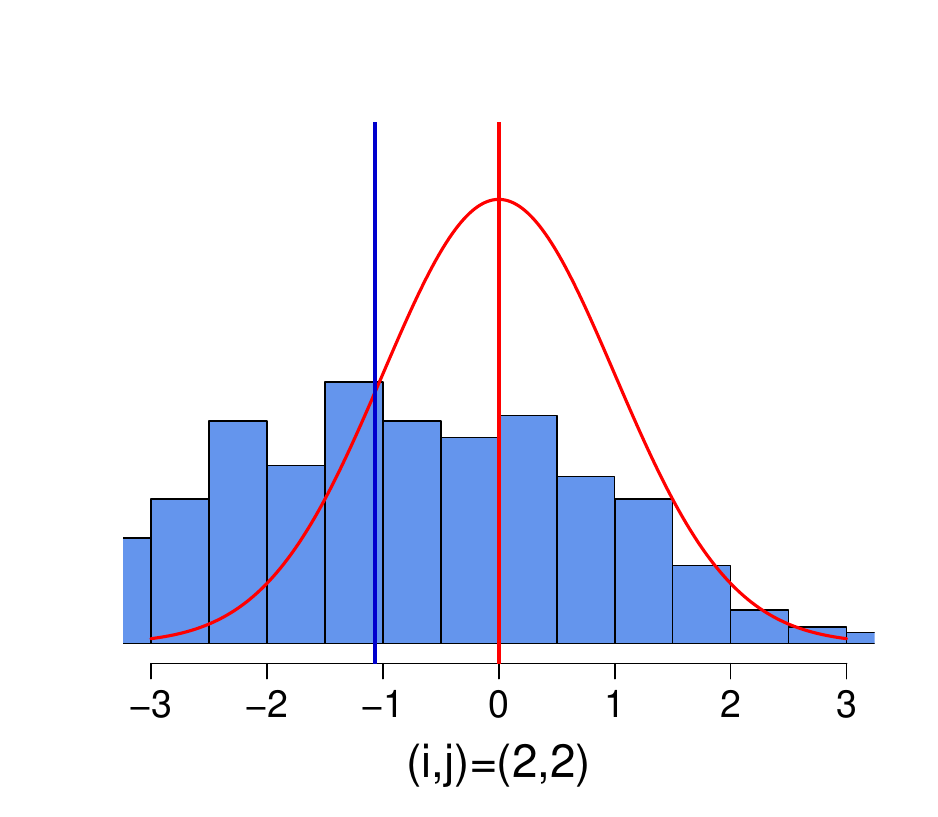}
    \end{minipage}
 \end{minipage}

  \caption*{$n=800, p=200$}
      \vspace{-0.43cm}
 \begin{minipage}{0.3\linewidth}
    \begin{minipage}{0.24\linewidth}
        \centering
        \includegraphics[width=\textwidth]{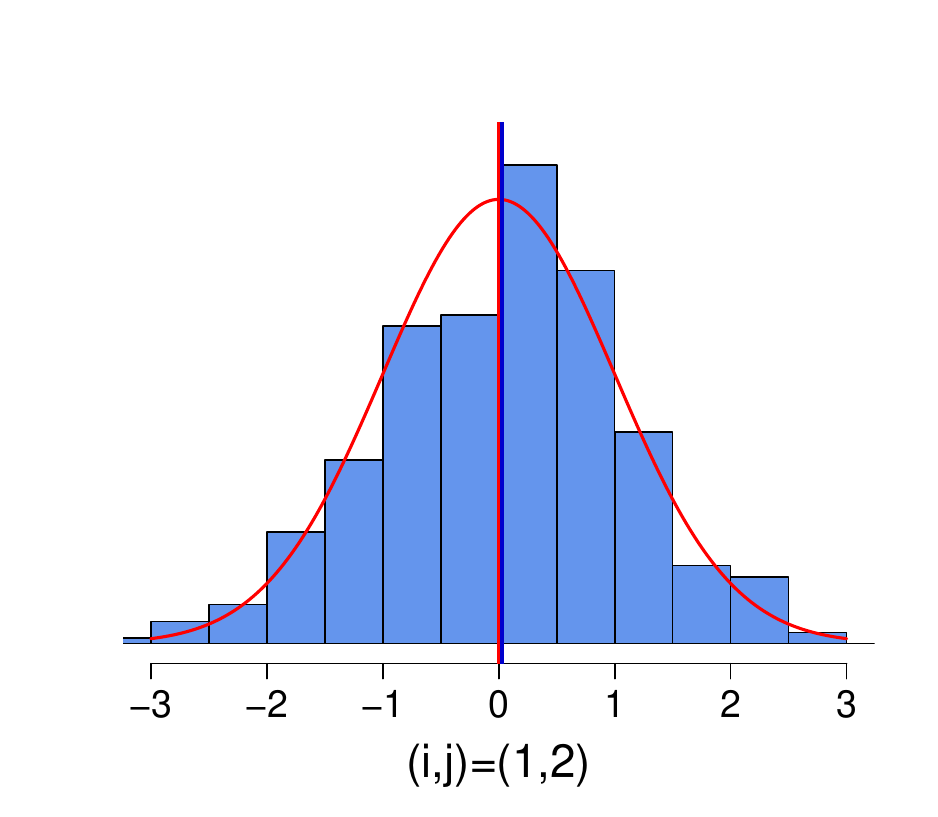}
    \end{minipage}
    \begin{minipage}{0.24\linewidth}
        \centering
        \includegraphics[width=\textwidth]{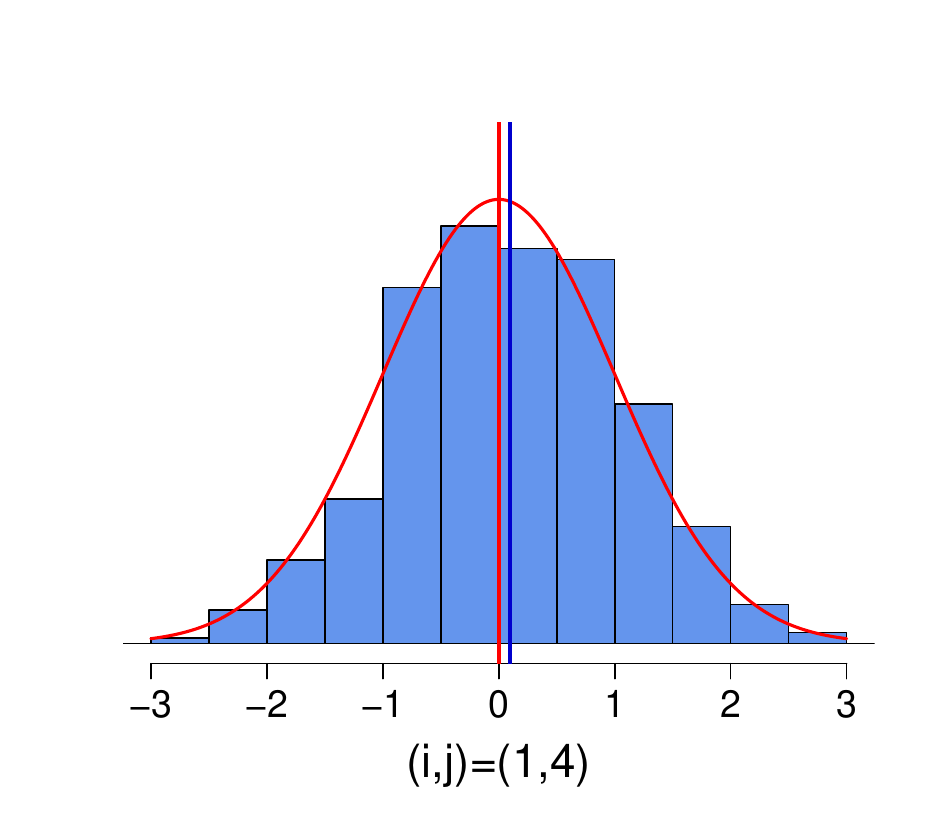}
    \end{minipage}
    \begin{minipage}{0.24\linewidth}
        \centering
        \includegraphics[width=\textwidth]{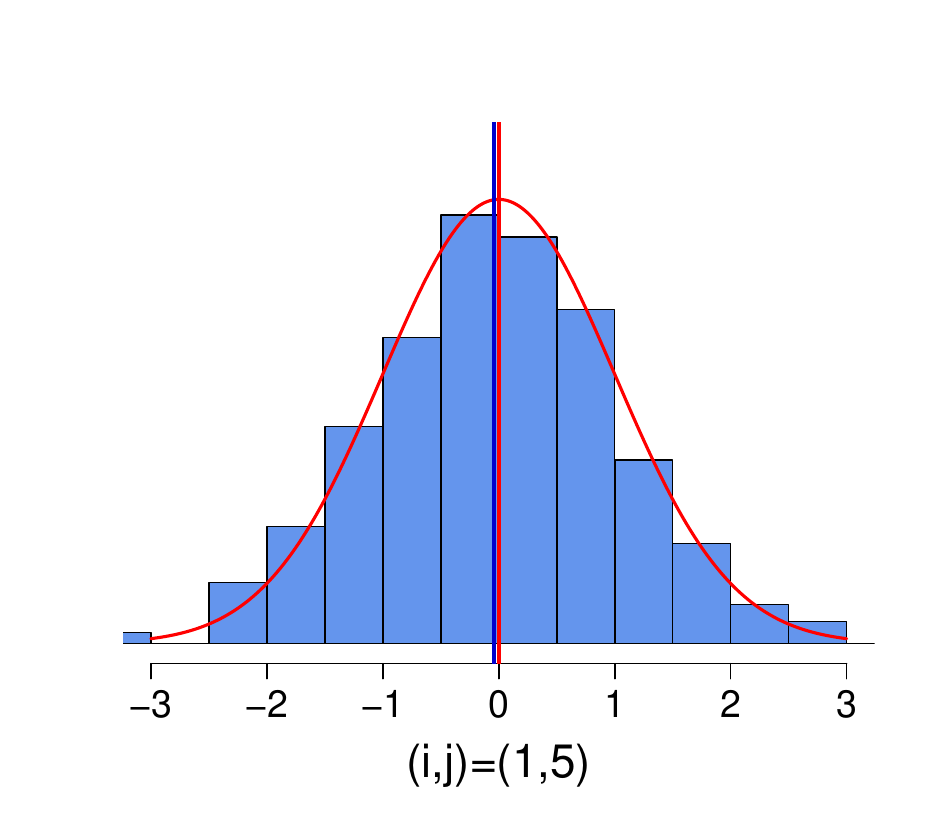}
    \end{minipage}
    \begin{minipage}{0.24\linewidth}
        \centering
        \includegraphics[width=\textwidth]{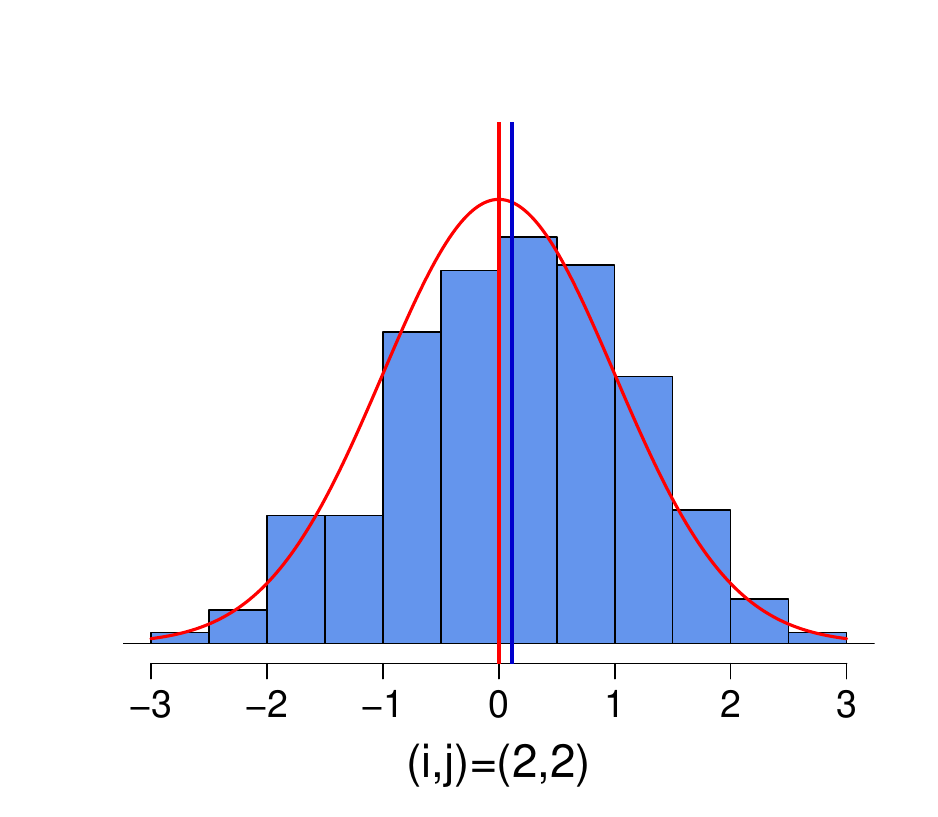}
    \end{minipage}
 \end{minipage} 
     \hspace{1cm}
 \begin{minipage}{0.3\linewidth}
    \begin{minipage}{0.24\linewidth}
        \centering
        \includegraphics[width=\textwidth]{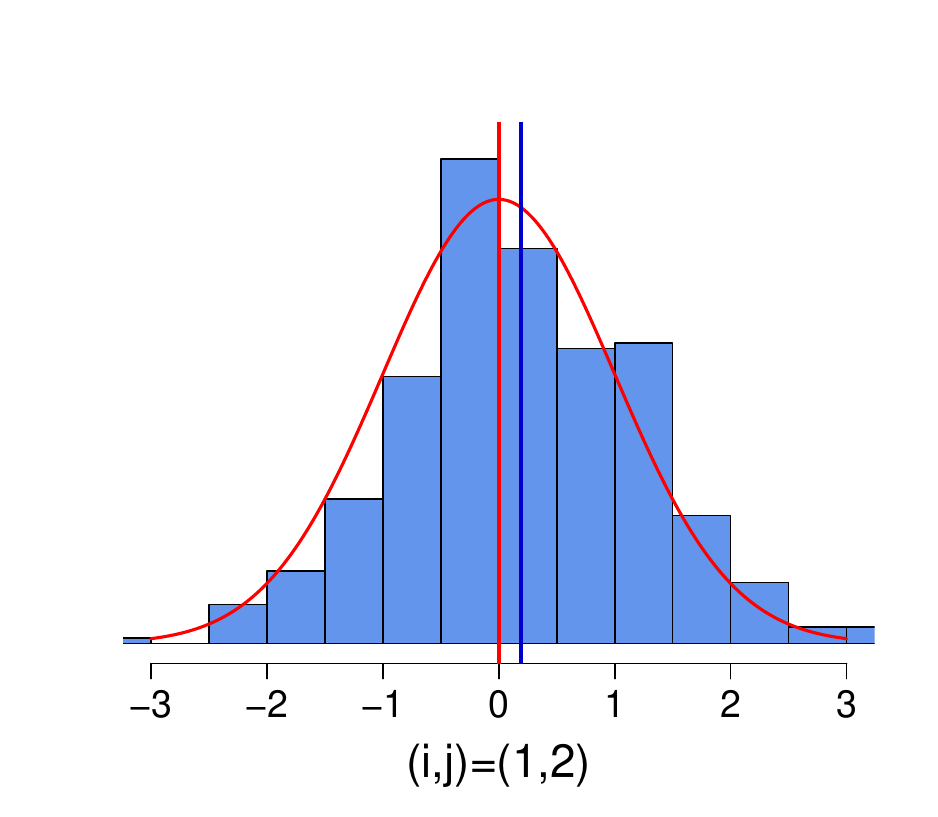}
    \end{minipage}
    \begin{minipage}{0.24\linewidth}
        \centering
        \includegraphics[width=\textwidth]{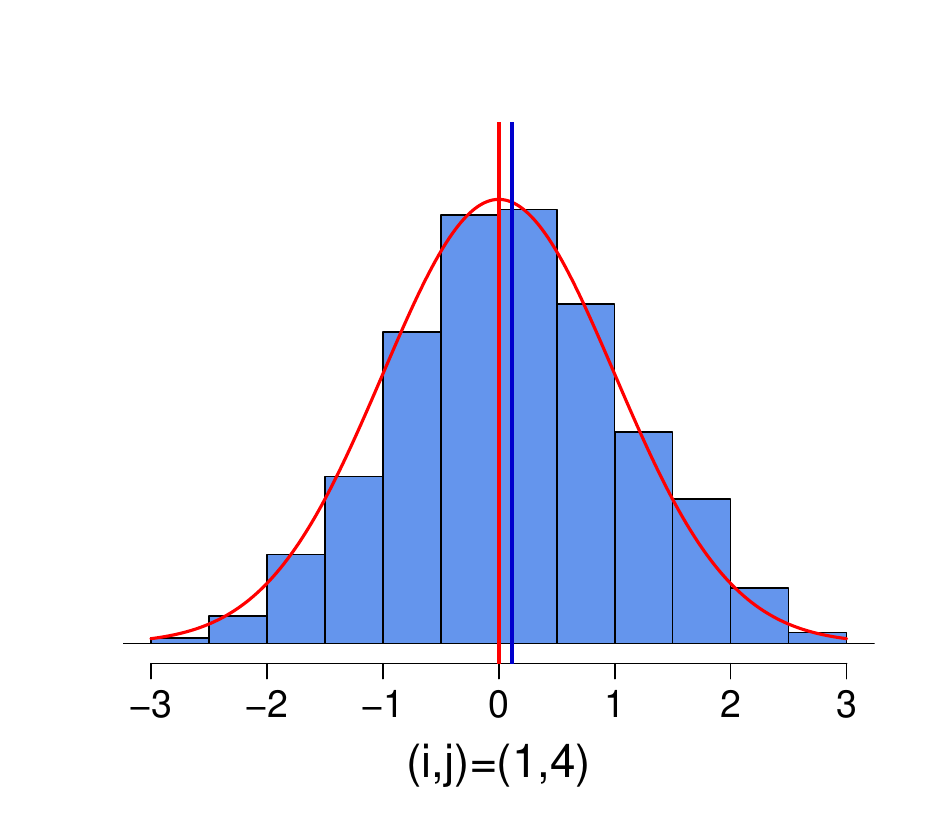}
    \end{minipage}
    \begin{minipage}{0.24\linewidth}
        \centering
        \includegraphics[width=\textwidth]{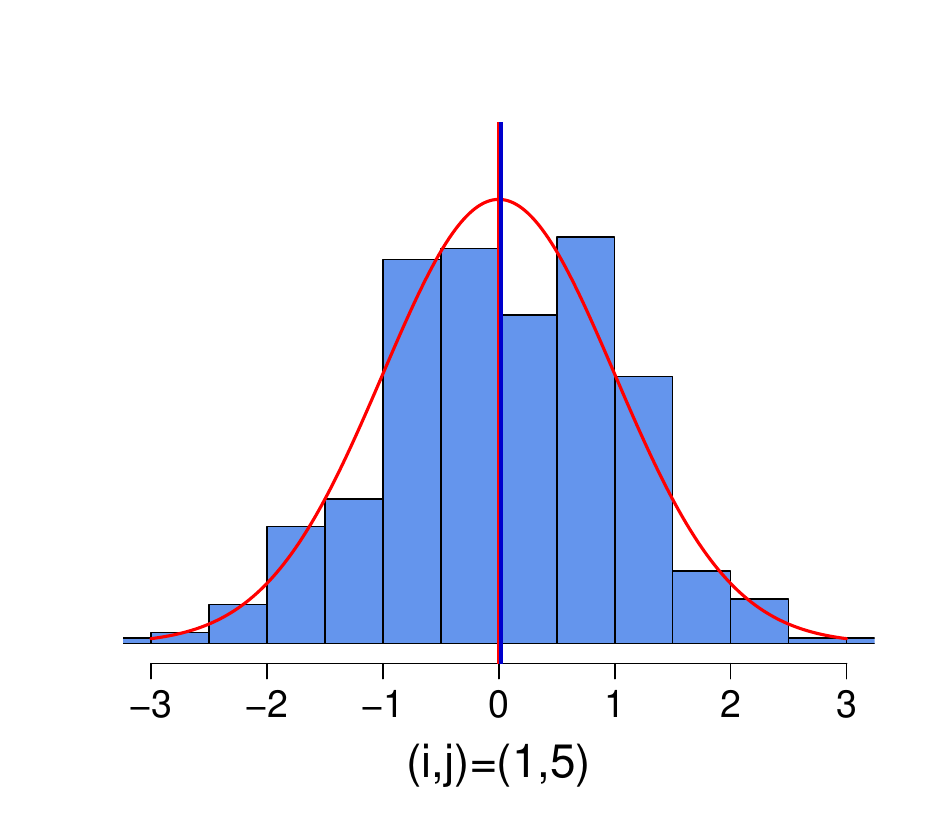}
    \end{minipage}
    \begin{minipage}{0.24\linewidth}
        \centering
        \includegraphics[width=\textwidth]{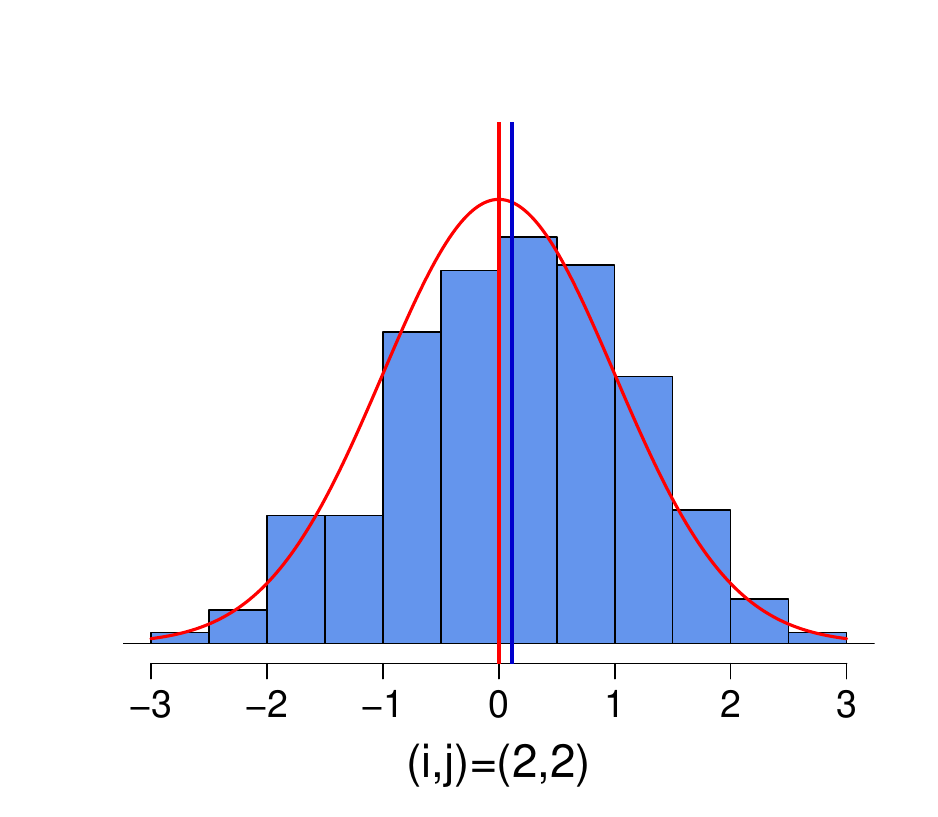}
    \end{minipage}
 \end{minipage}   
      \hspace{1cm}
 \begin{minipage}{0.3\linewidth}
    \begin{minipage}{0.24\linewidth}
        \centering
        \includegraphics[width=\textwidth]{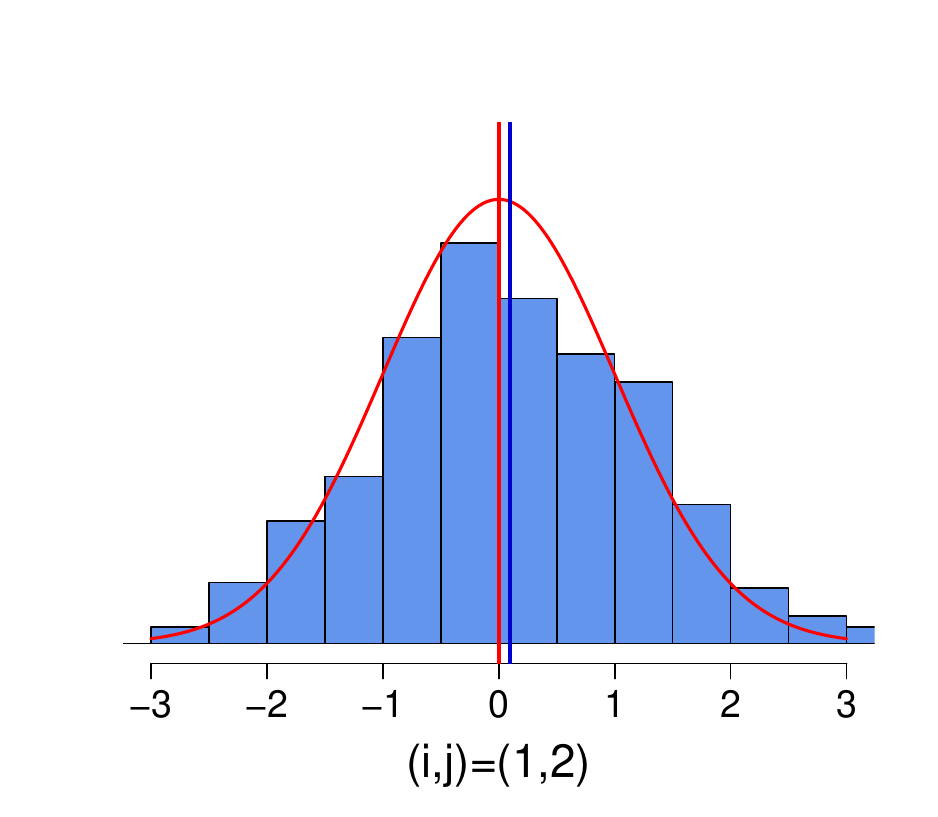}
    \end{minipage}
    \begin{minipage}{0.24\linewidth}
        \centering
        \includegraphics[width=\textwidth]{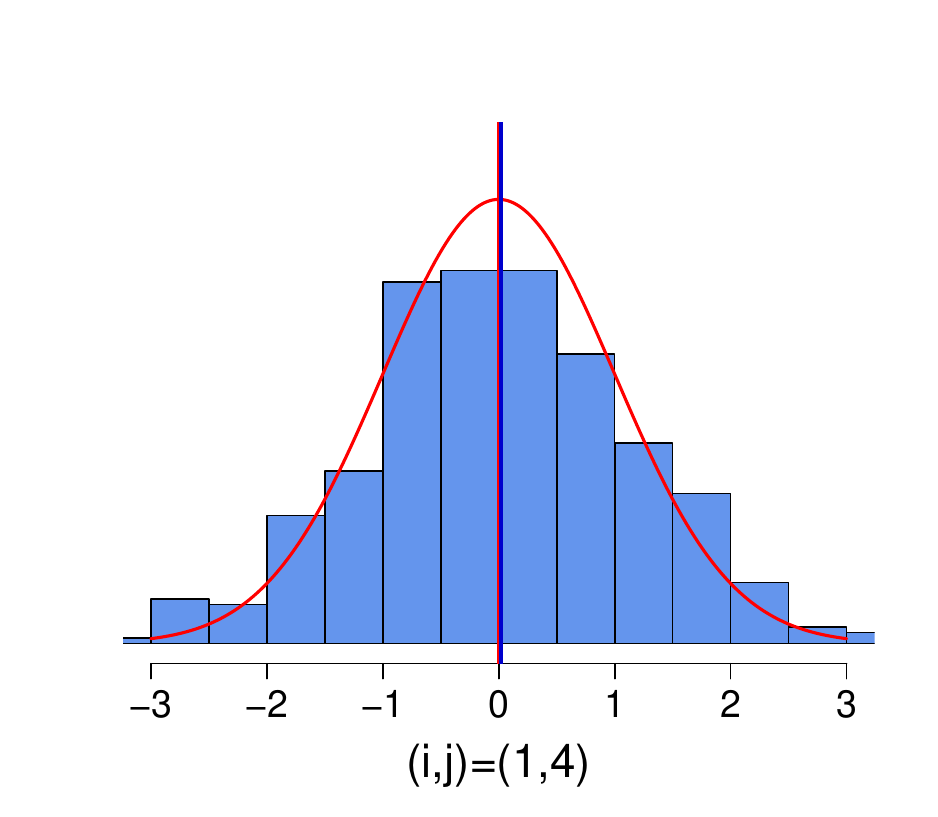}
    \end{minipage}
    \begin{minipage}{0.24\linewidth}
        \centering
        \includegraphics[width=\textwidth]{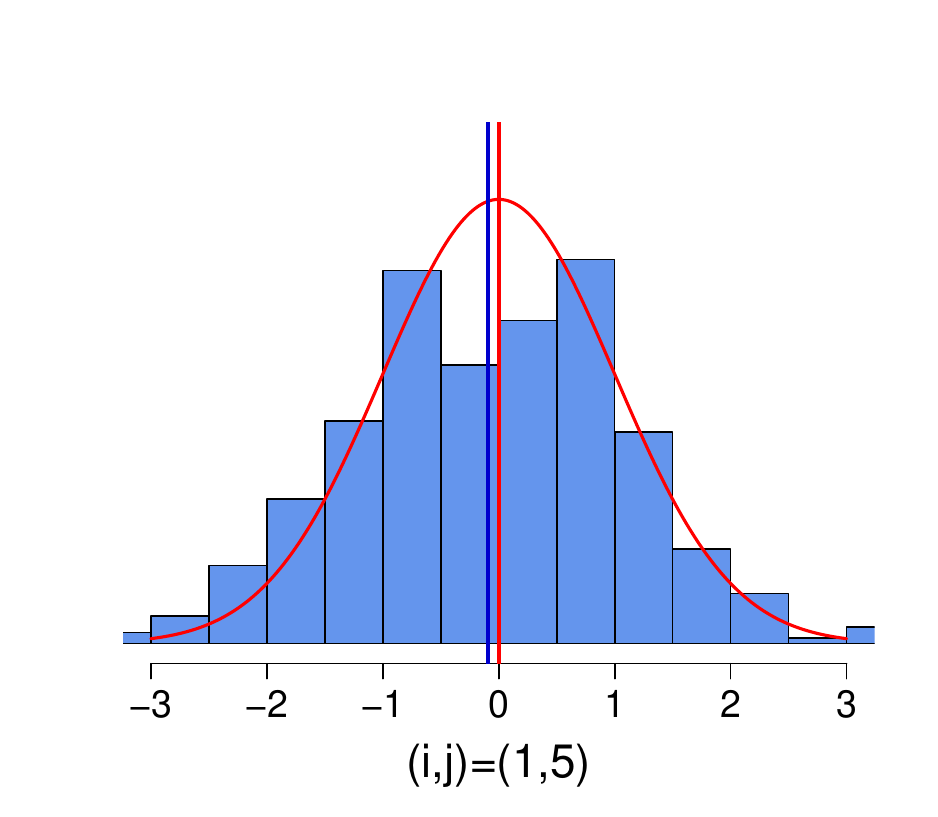}
    \end{minipage}
    \begin{minipage}{0.24\linewidth}
        \centering
        \includegraphics[width=\textwidth]{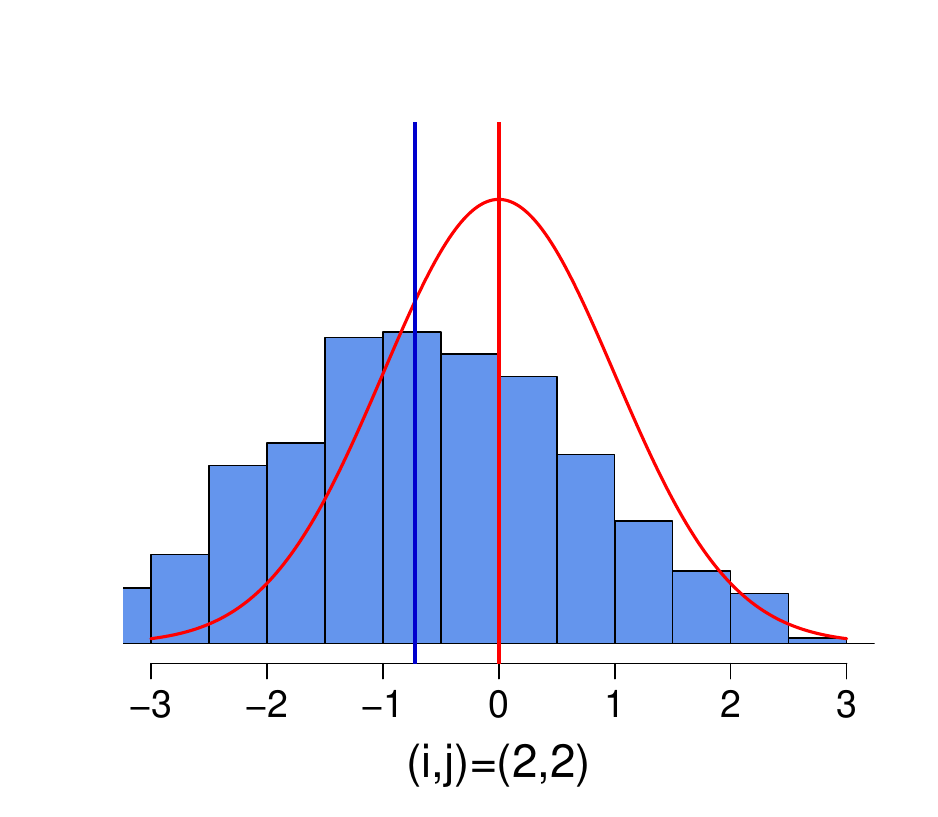}
    \end{minipage}
     \end{minipage}   
     
 \caption*{$n=200, p=400$}
     \vspace{-0.43cm}
 \begin{minipage}{0.3\linewidth}
    \begin{minipage}{0.24\linewidth}
        \centering
        \includegraphics[width=\textwidth]{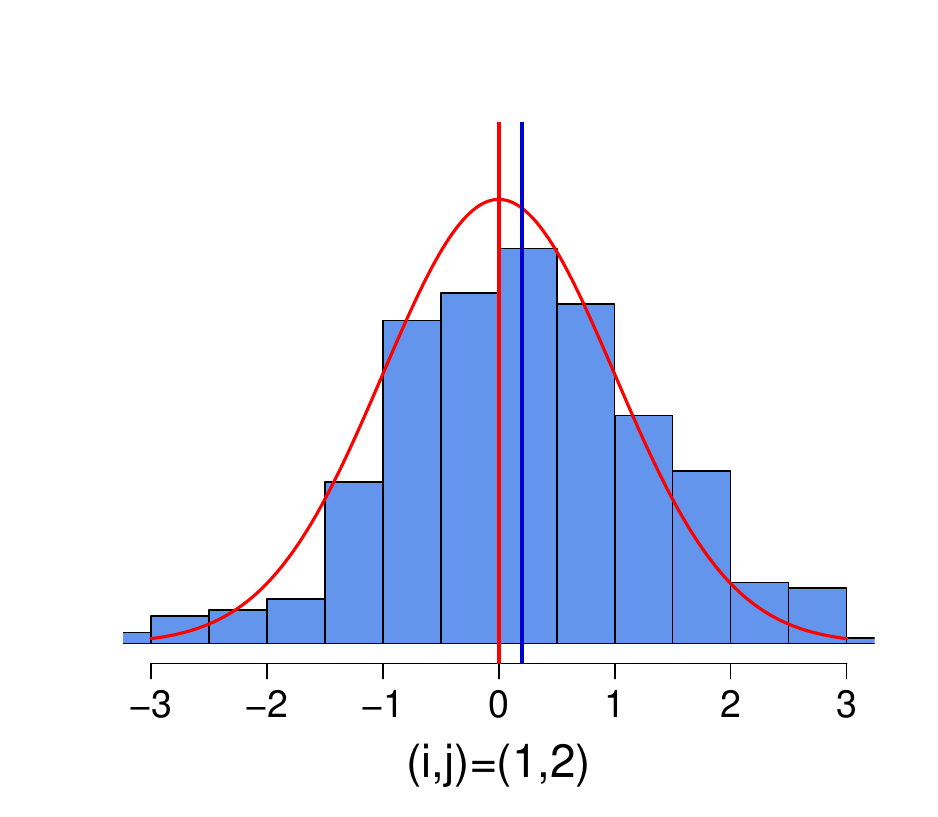}
    \end{minipage}
    \begin{minipage}{0.24\linewidth}
        \centering
        \includegraphics[width=\textwidth]{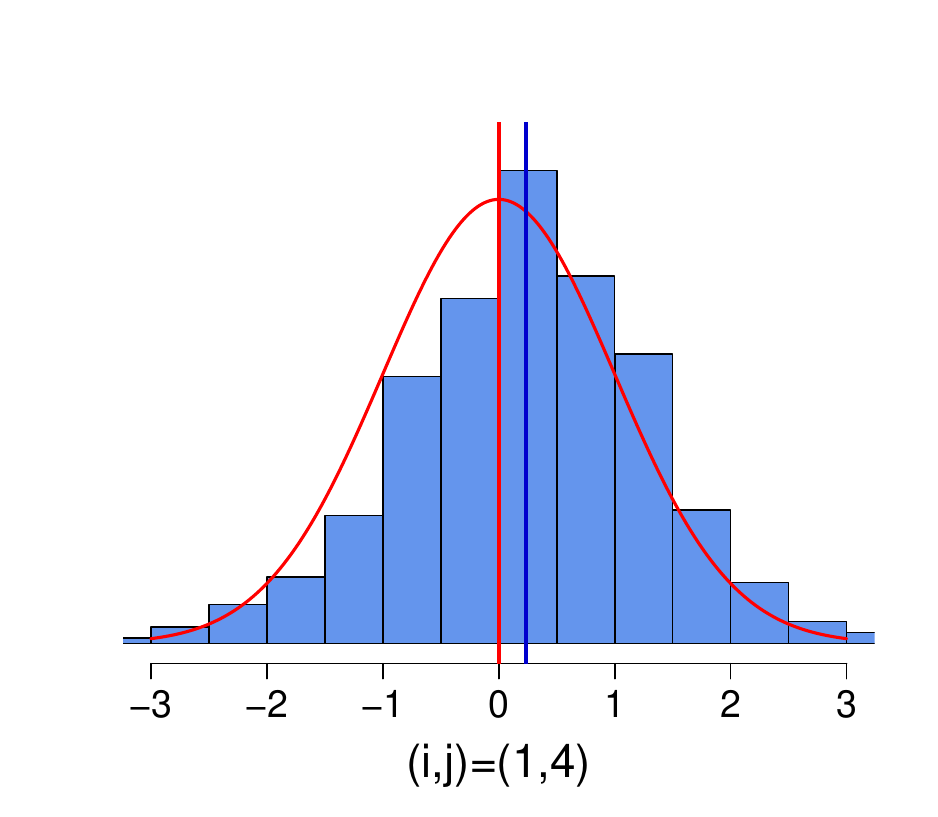}
    \end{minipage}
    \begin{minipage}{0.24\linewidth}
        \centering
        \includegraphics[width=\textwidth]{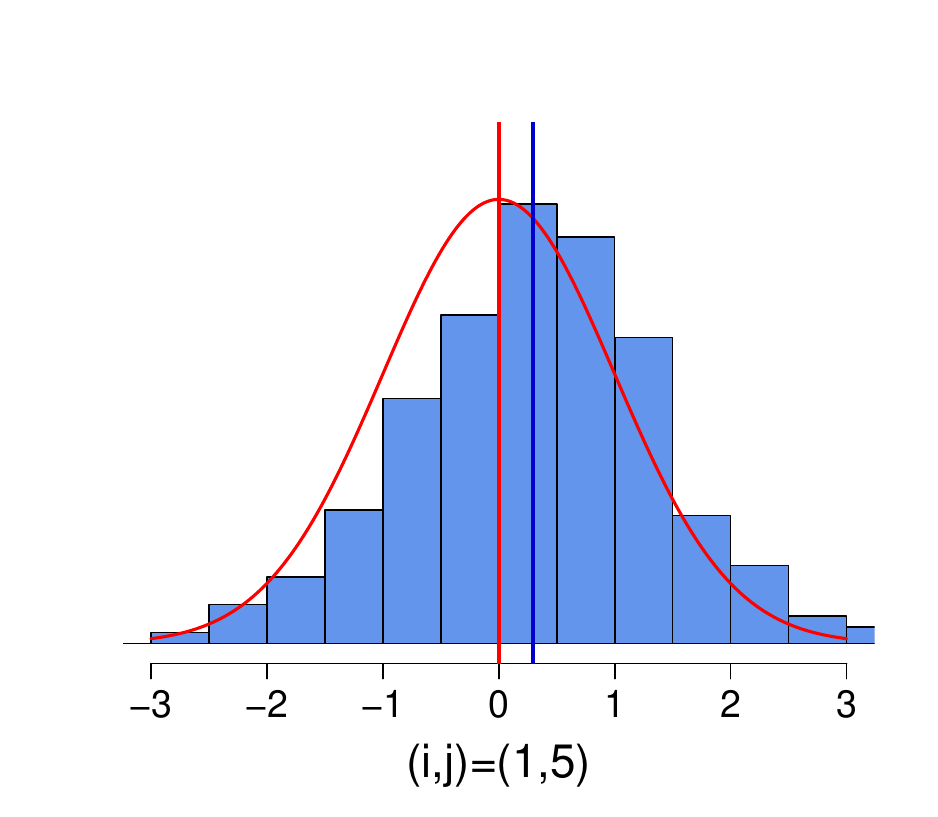}
    \end{minipage}
    \begin{minipage}{0.24\linewidth}
        \centering
        \includegraphics[width=\textwidth]{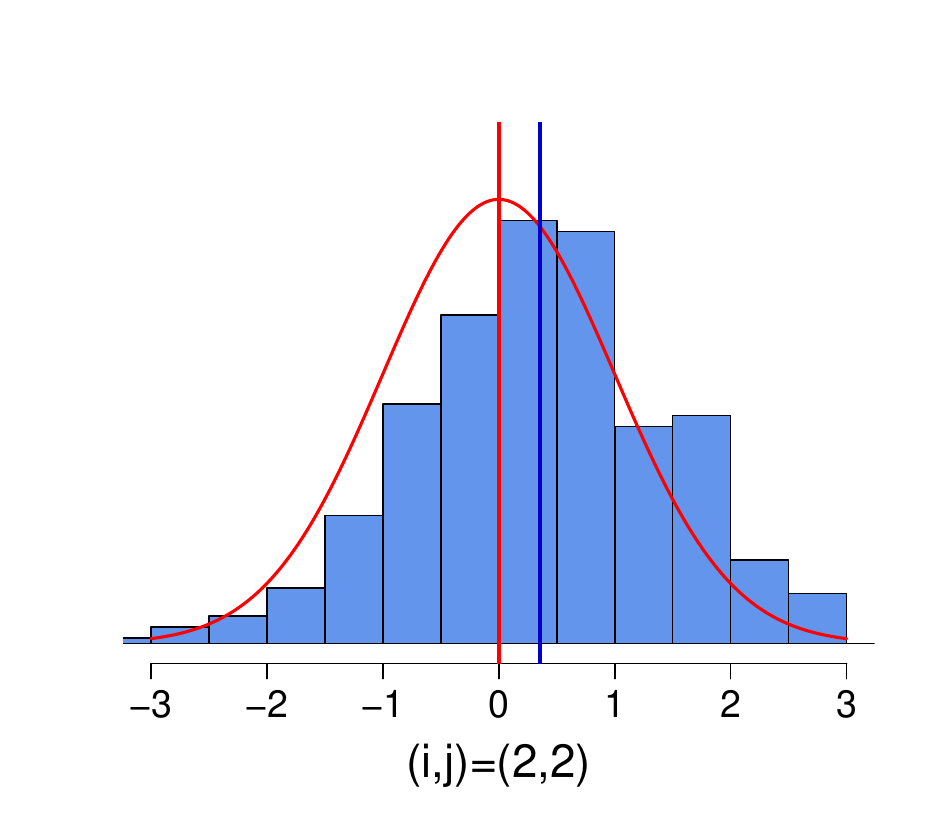}
    \end{minipage}
 \end{minipage}
 \hspace{1cm}
 \begin{minipage}{0.3\linewidth}
    \begin{minipage}{0.24\linewidth}
        \centering
        \includegraphics[width=\textwidth]{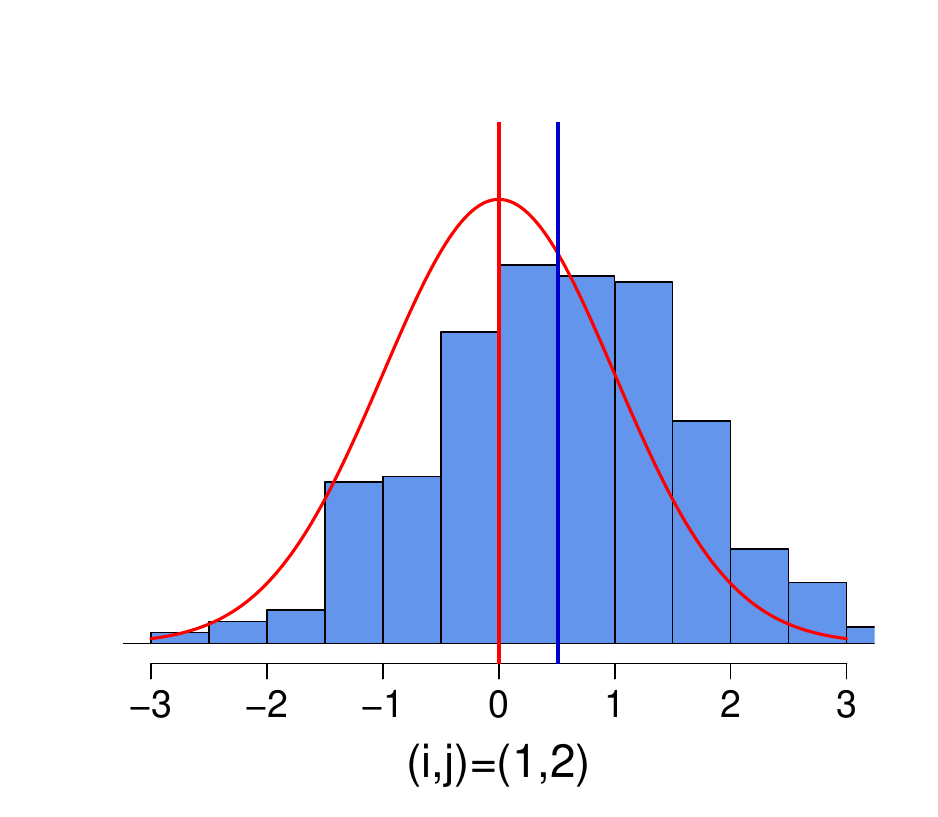}
    \end{minipage}
    \begin{minipage}{0.24\linewidth}
        \centering
        \includegraphics[width=\textwidth]{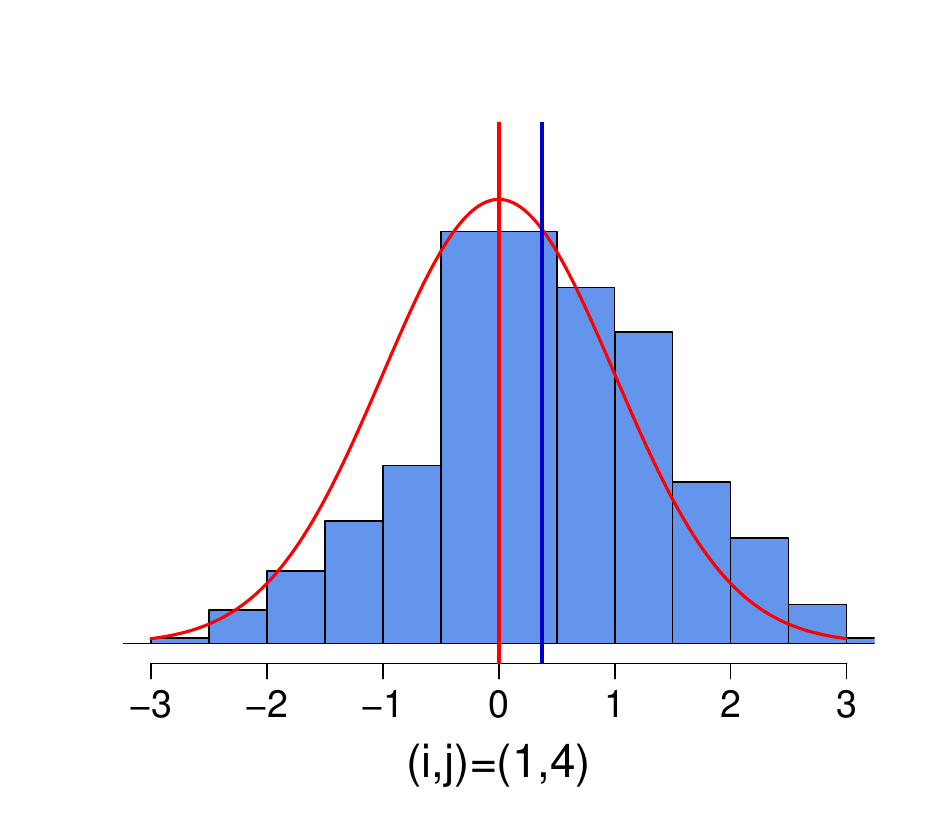}
    \end{minipage}
    \begin{minipage}{0.24\linewidth}
        \centering
        \includegraphics[width=\textwidth]{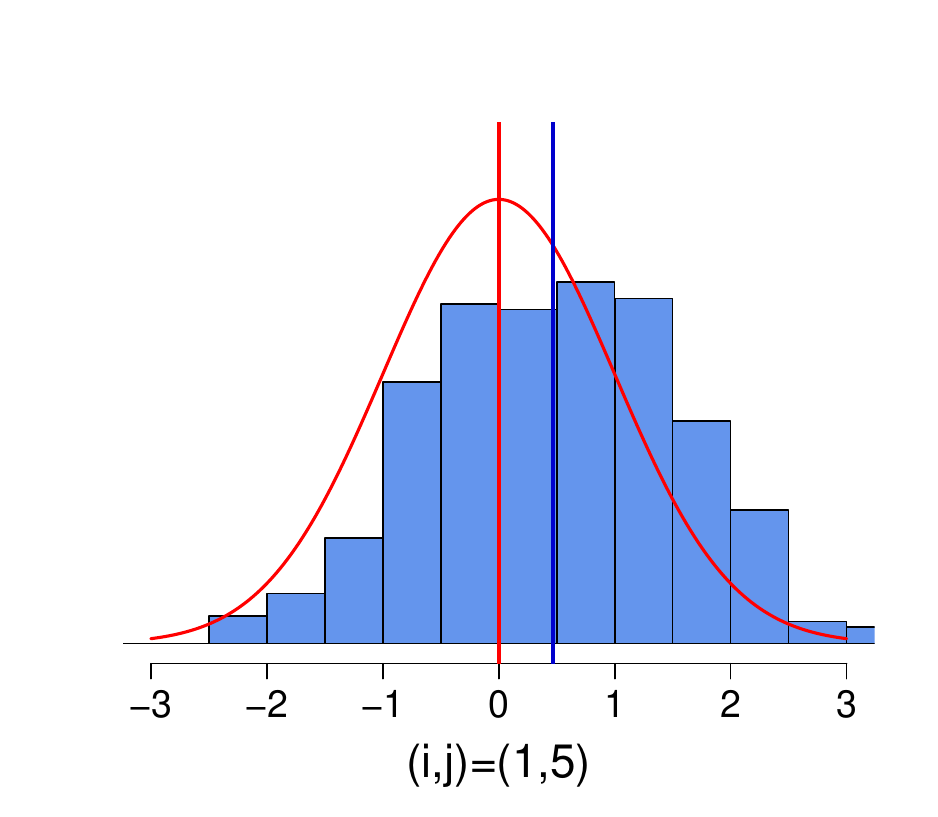}
    \end{minipage}
    \begin{minipage}{0.24\linewidth}
        \centering
        \includegraphics[width=\textwidth]{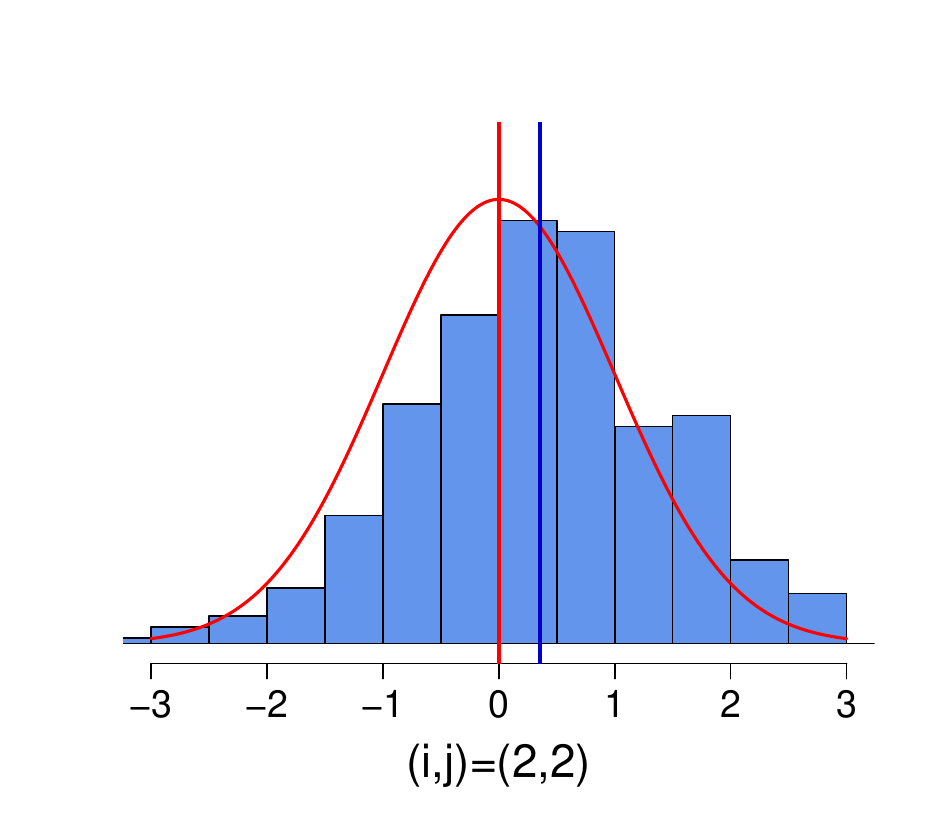}
    \end{minipage}    
 \end{minipage}
  \hspace{1cm}
 \begin{minipage}{0.3\linewidth}
     \begin{minipage}{0.24\linewidth}
        \centering
        \includegraphics[width=\textwidth]{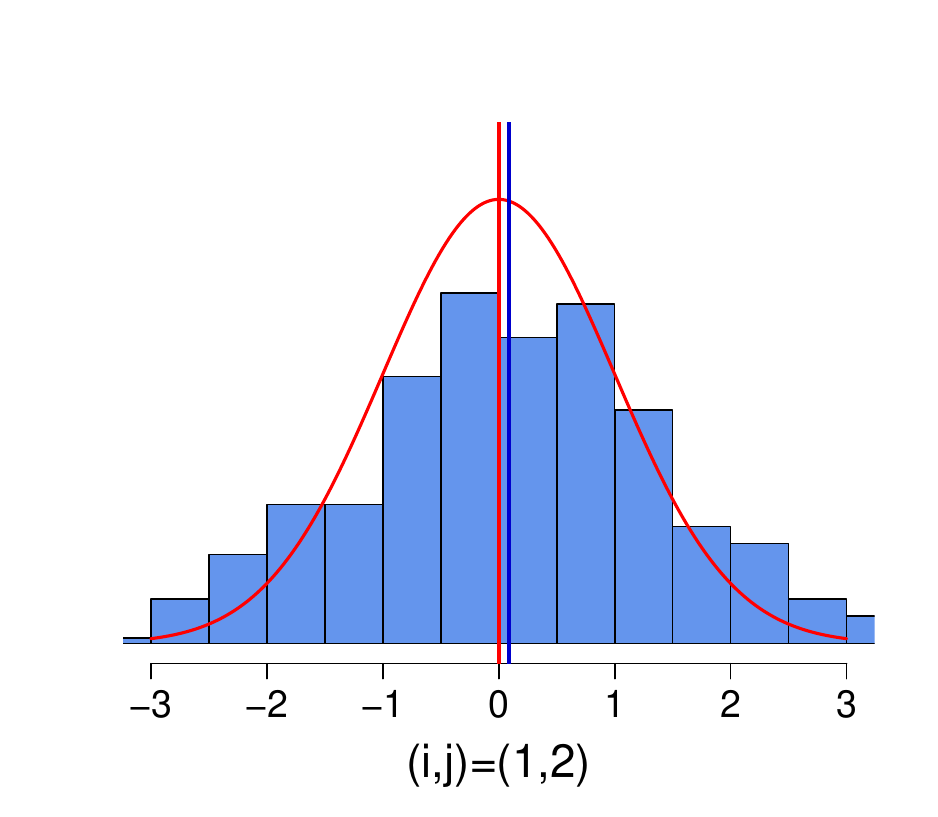}
    \end{minipage}
    \begin{minipage}{0.24\linewidth}
        \centering
        \includegraphics[width=\textwidth]{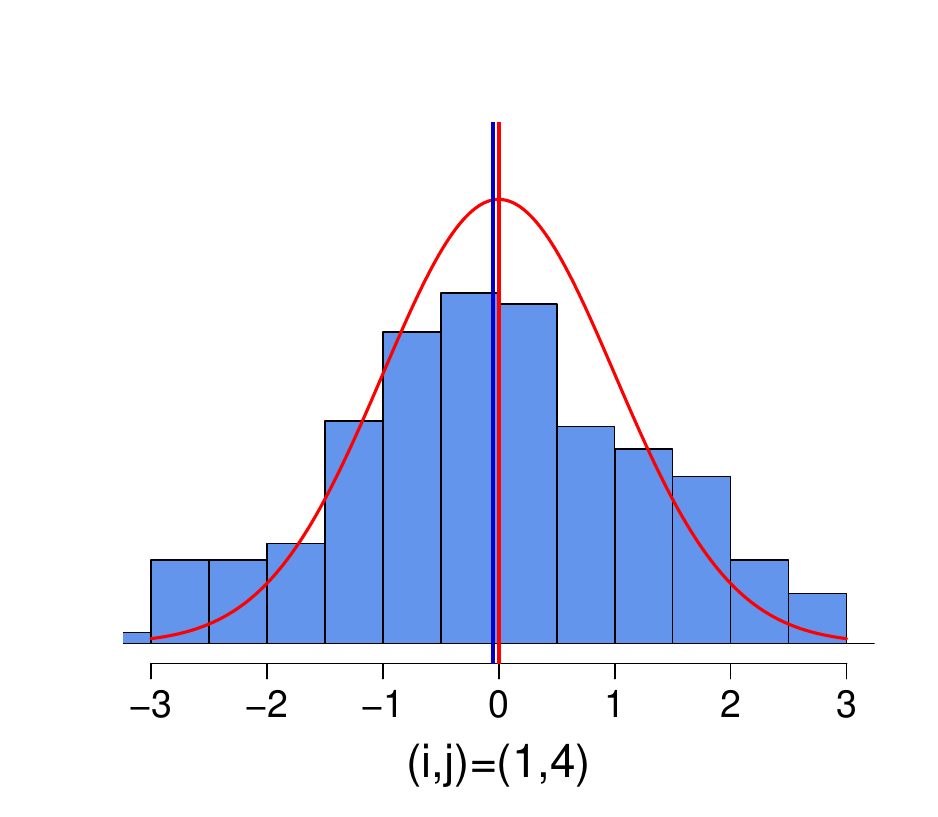}
    \end{minipage}
    \begin{minipage}{0.24\linewidth}
        \centering
        \includegraphics[width=\textwidth]{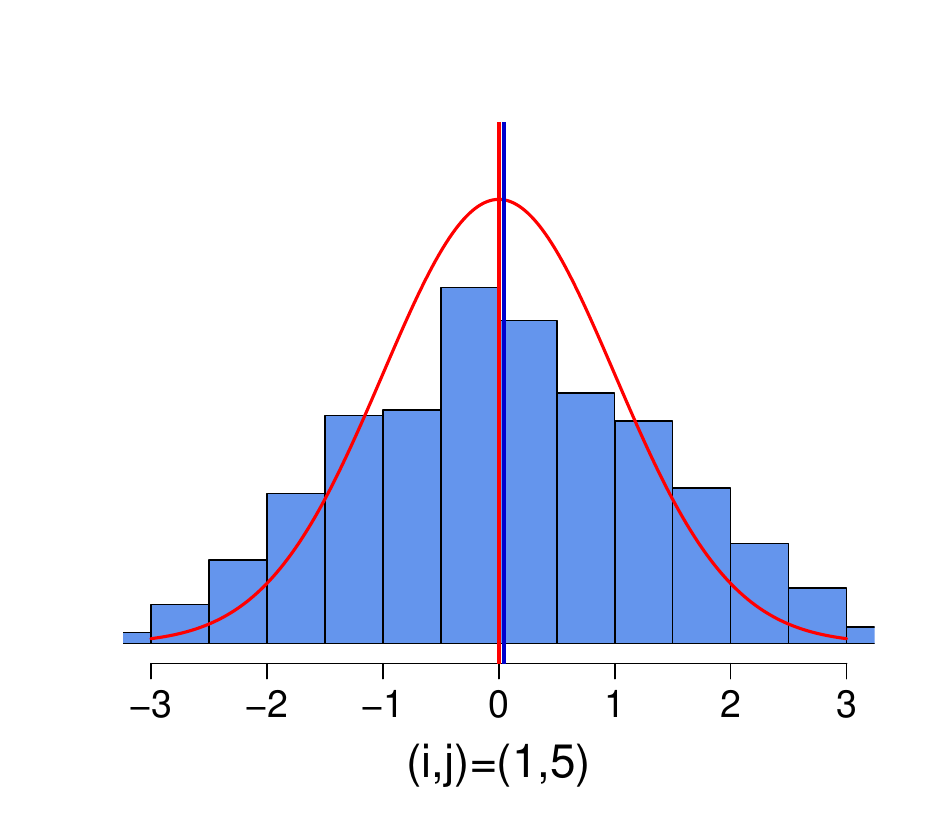}
    \end{minipage}
    \begin{minipage}{0.24\linewidth}
        \centering
        \includegraphics[width=\textwidth]{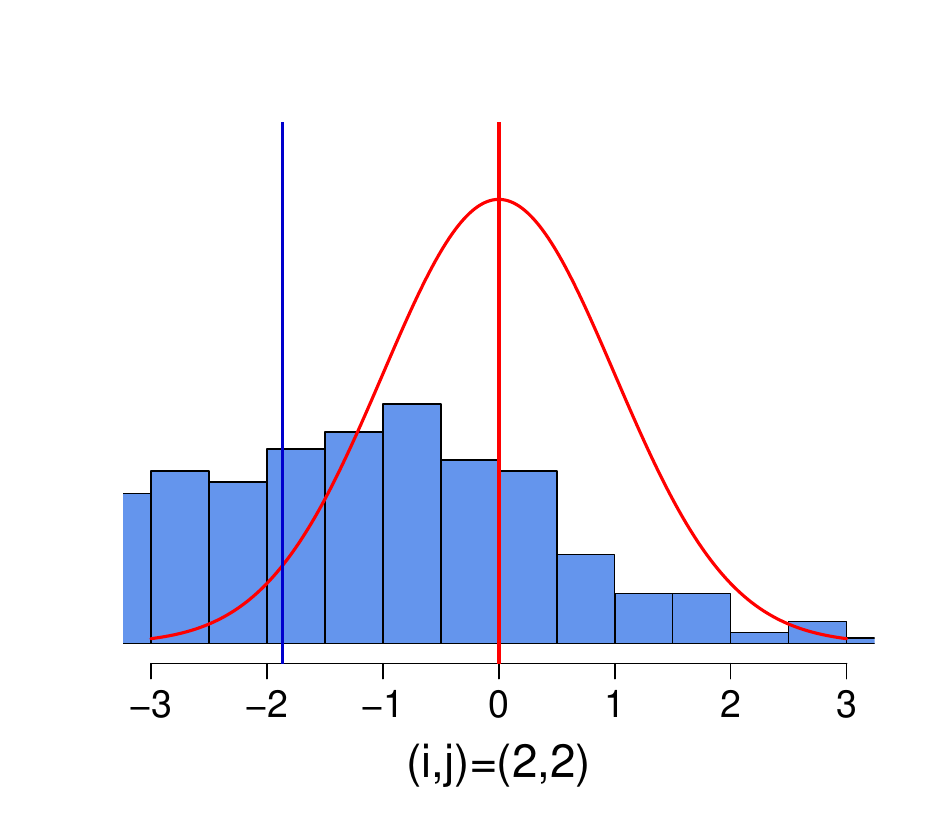}
    \end{minipage}
 \end{minipage}

  \caption*{$n=400, p=400$}
      \vspace{-0.43cm}
 \begin{minipage}{0.3\linewidth}
    \begin{minipage}{0.24\linewidth}
        \centering
        \includegraphics[width=\textwidth]{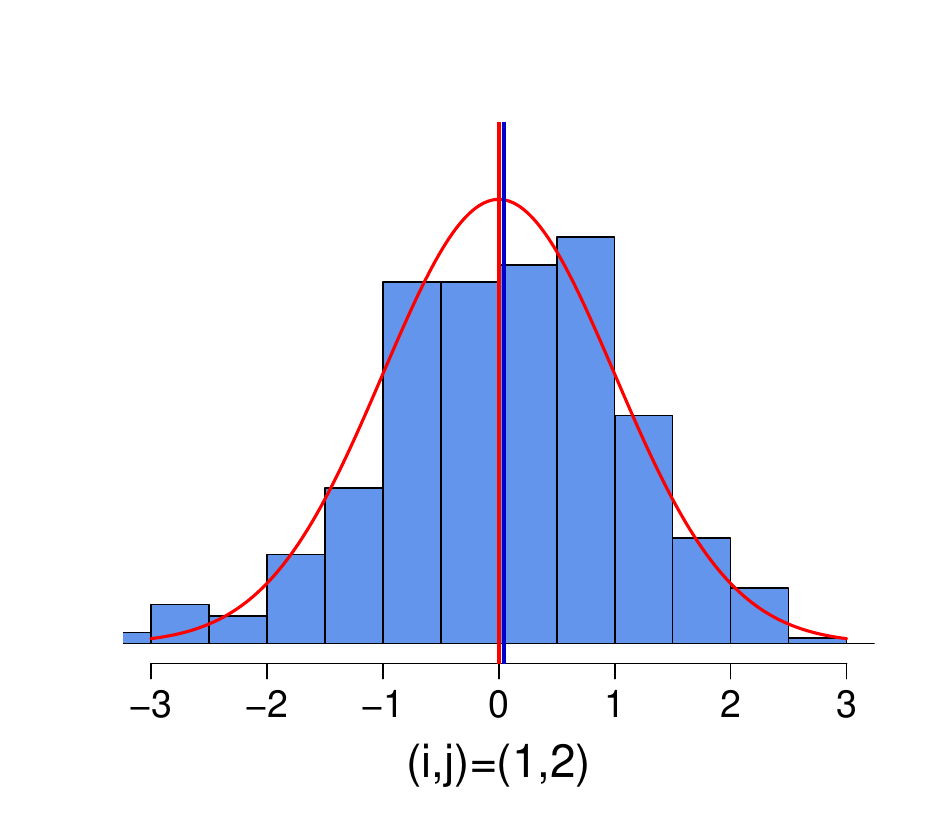}
    \end{minipage}
    \begin{minipage}{0.24\linewidth}
        \centering
        \includegraphics[width=\textwidth]{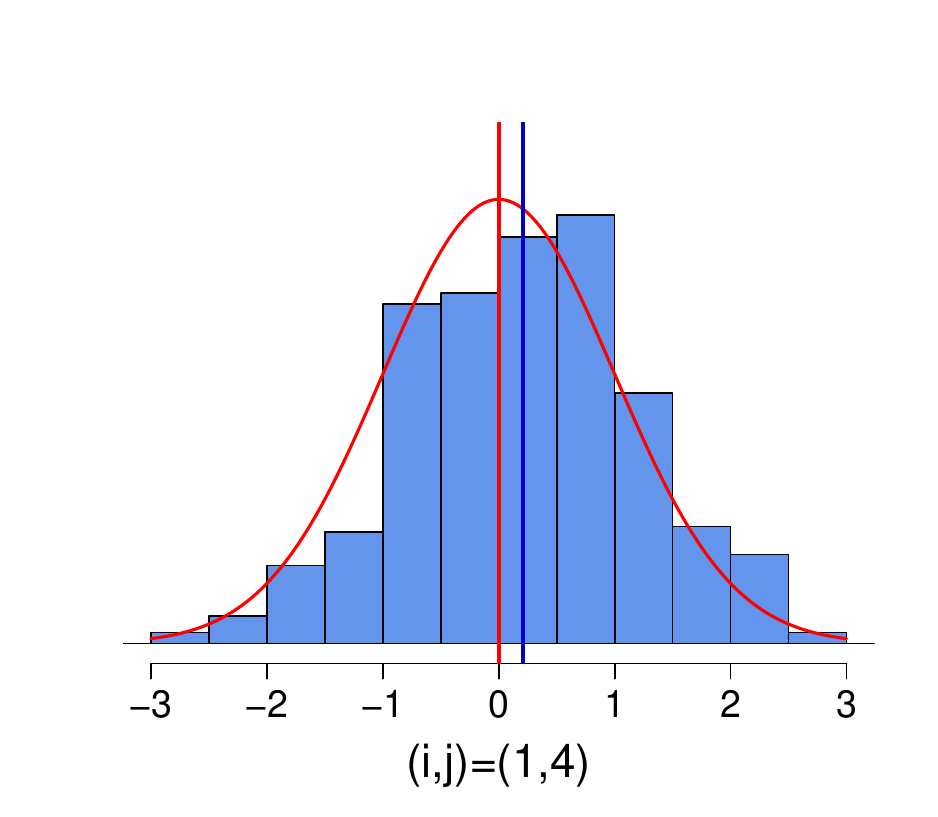}
    \end{minipage}
    \begin{minipage}{0.24\linewidth}
        \centering
        \includegraphics[width=\textwidth]{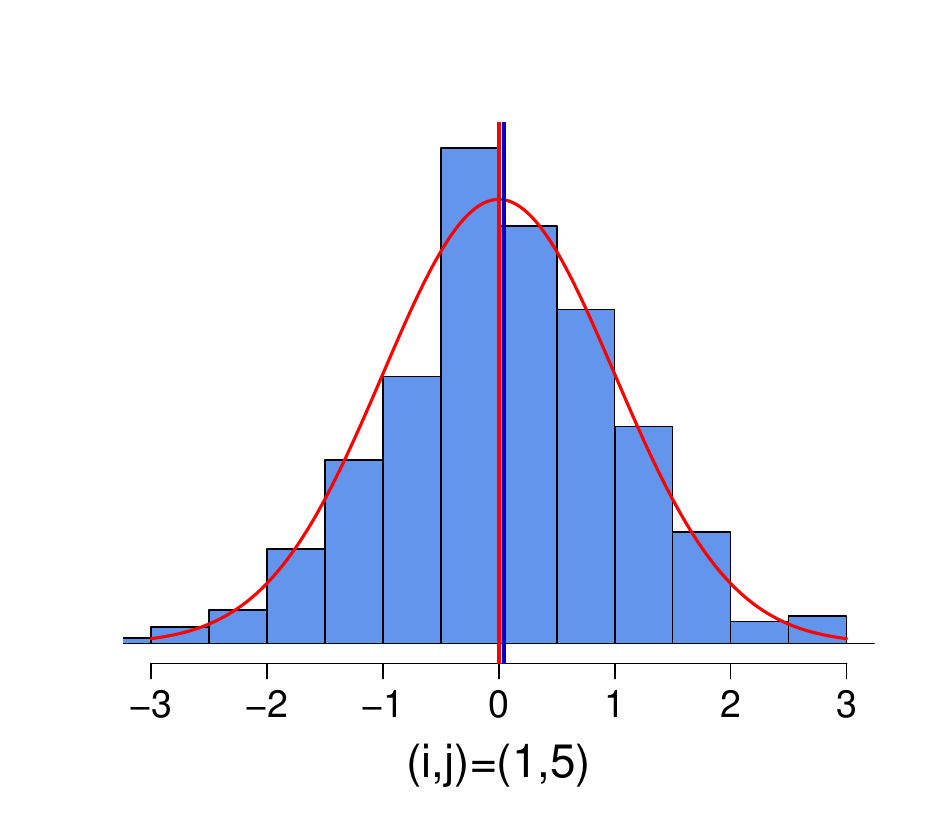}
    \end{minipage}
    \begin{minipage}{0.24\linewidth}
        \centering
        \includegraphics[width=\textwidth]{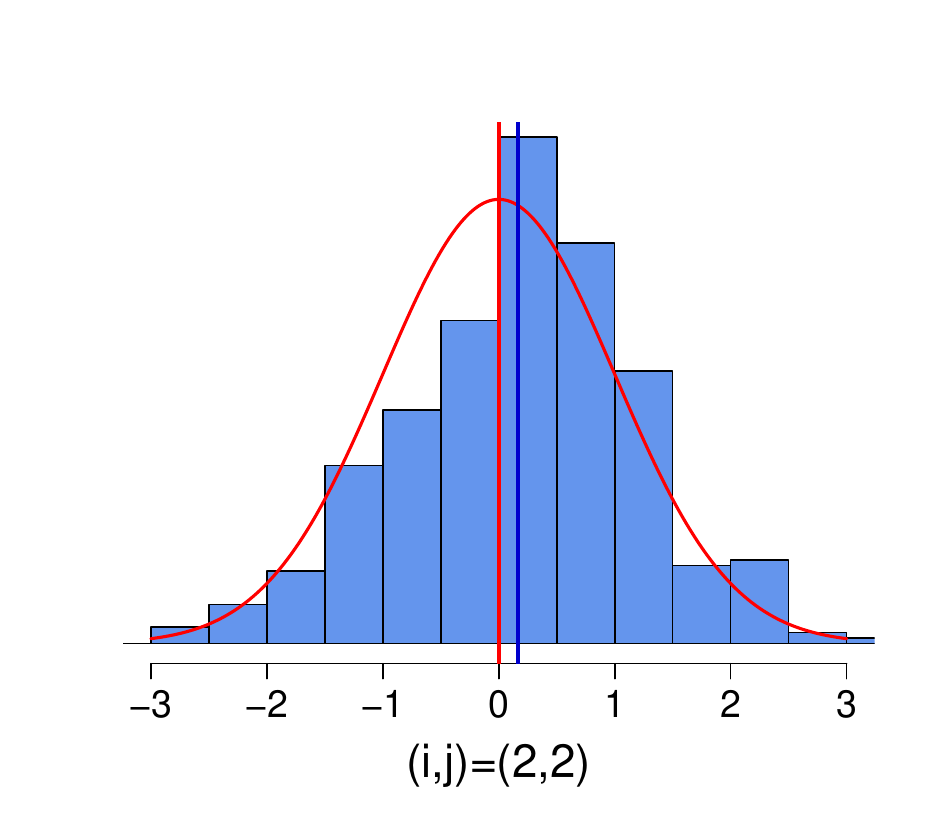}
    \end{minipage}
 \end{minipage}  
     \hspace{1cm}
 \begin{minipage}{0.3\linewidth}
    \begin{minipage}{0.24\linewidth}
        \centering
        \includegraphics[width=\textwidth]{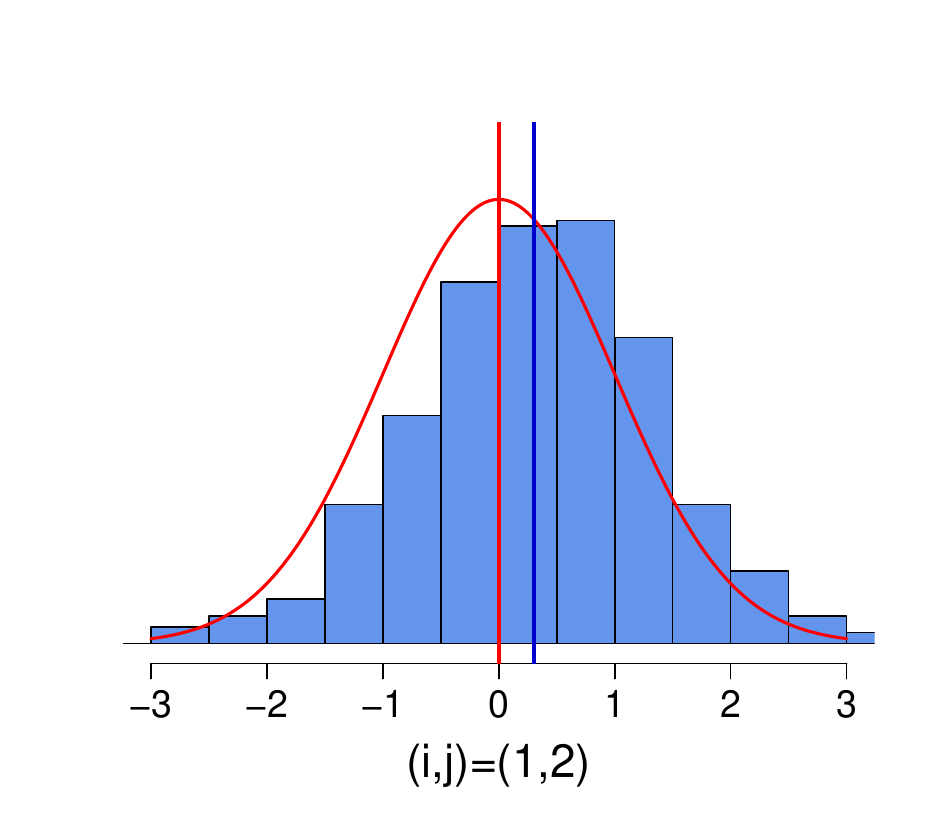}
    \end{minipage}
    \begin{minipage}{0.24\linewidth}
        \centering
        \includegraphics[width=\textwidth]{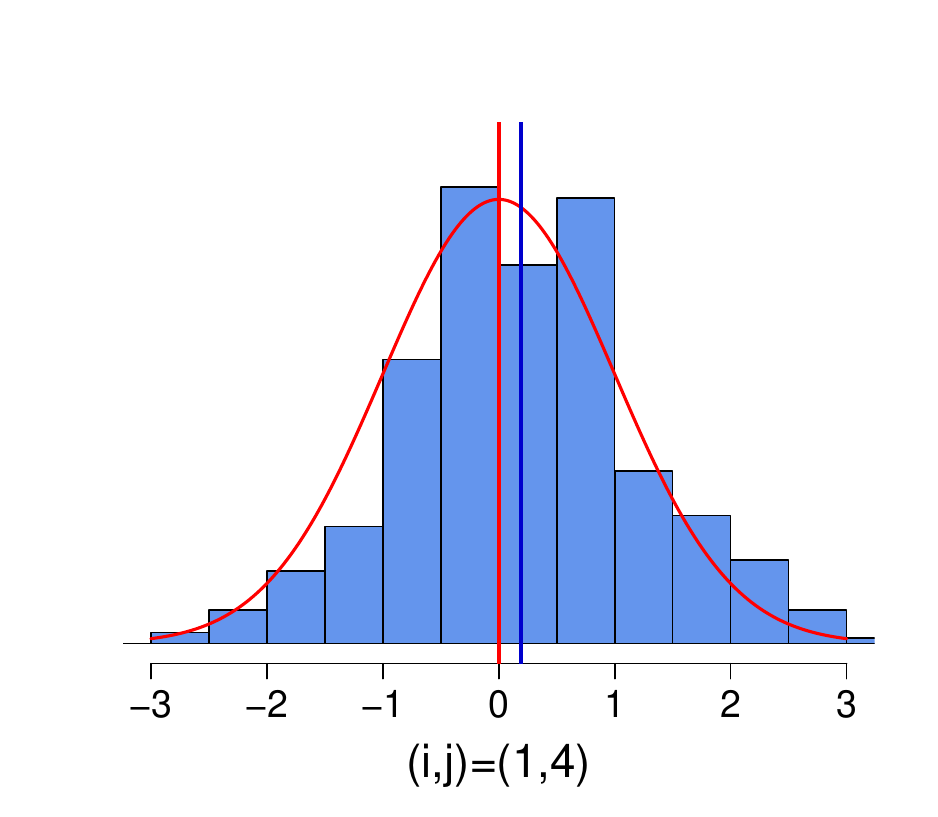}
    \end{minipage}
    \begin{minipage}{0.24\linewidth}
        \centering
        \includegraphics[width=\textwidth]{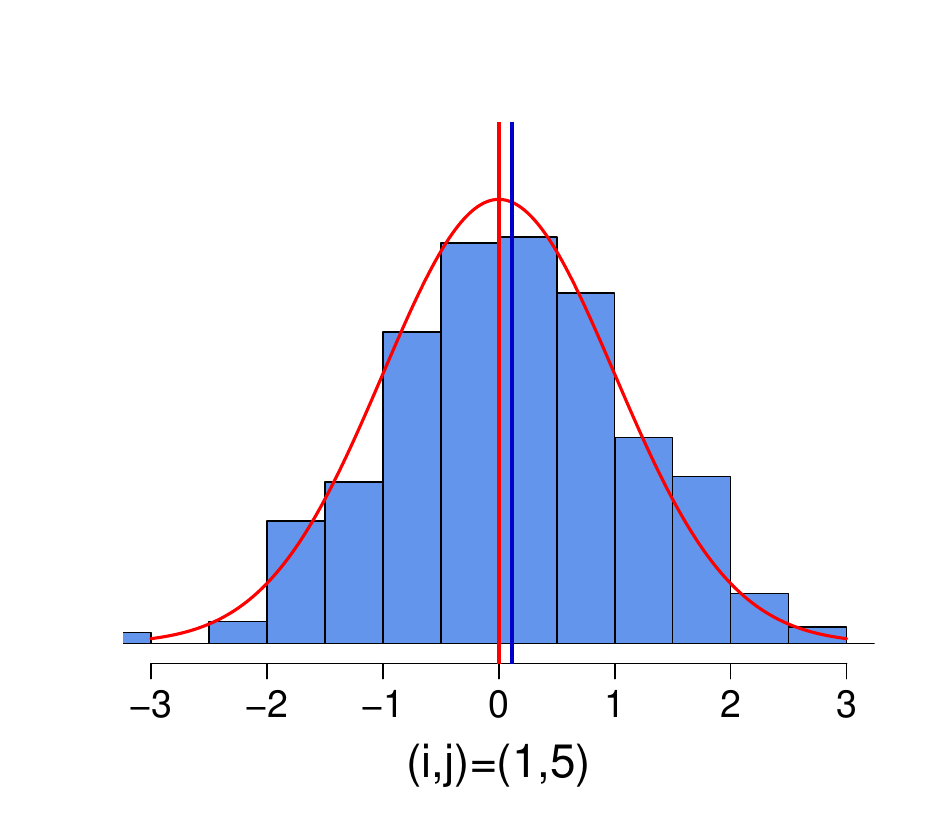}
    \end{minipage}
    \begin{minipage}{0.24\linewidth}
        \centering
        \includegraphics[width=\textwidth]{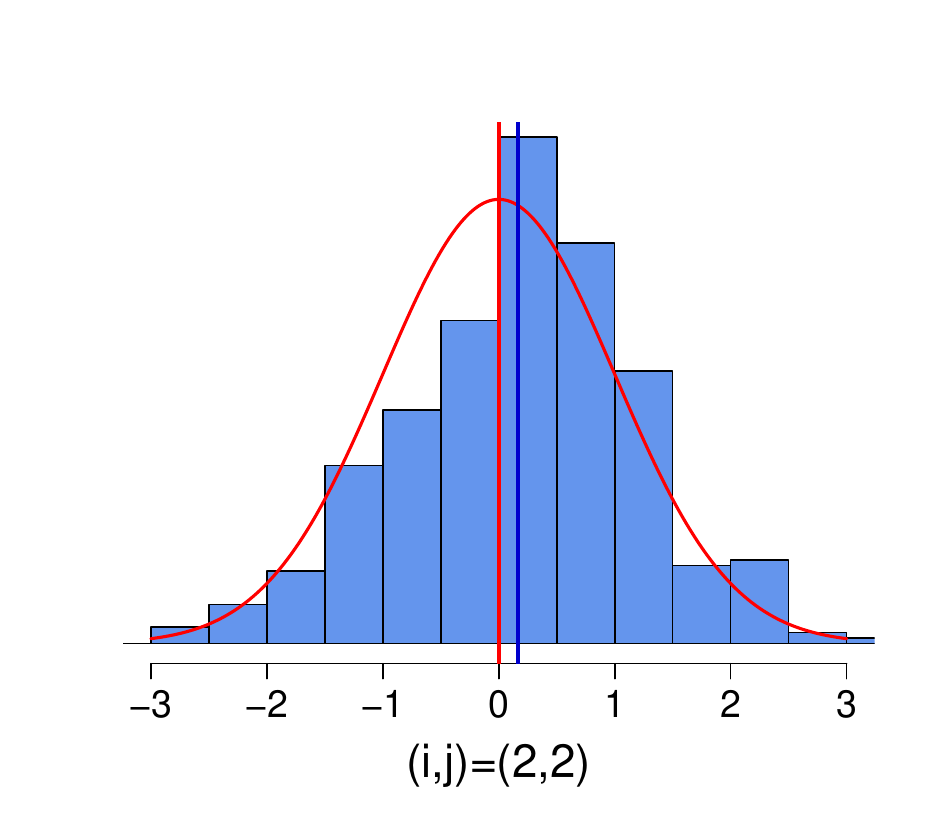}
    \end{minipage}
  \end{minipage}  
    \hspace{1cm}
 \begin{minipage}{0.3\linewidth}
    \begin{minipage}{0.24\linewidth}
        \centering
        \includegraphics[width=\textwidth]{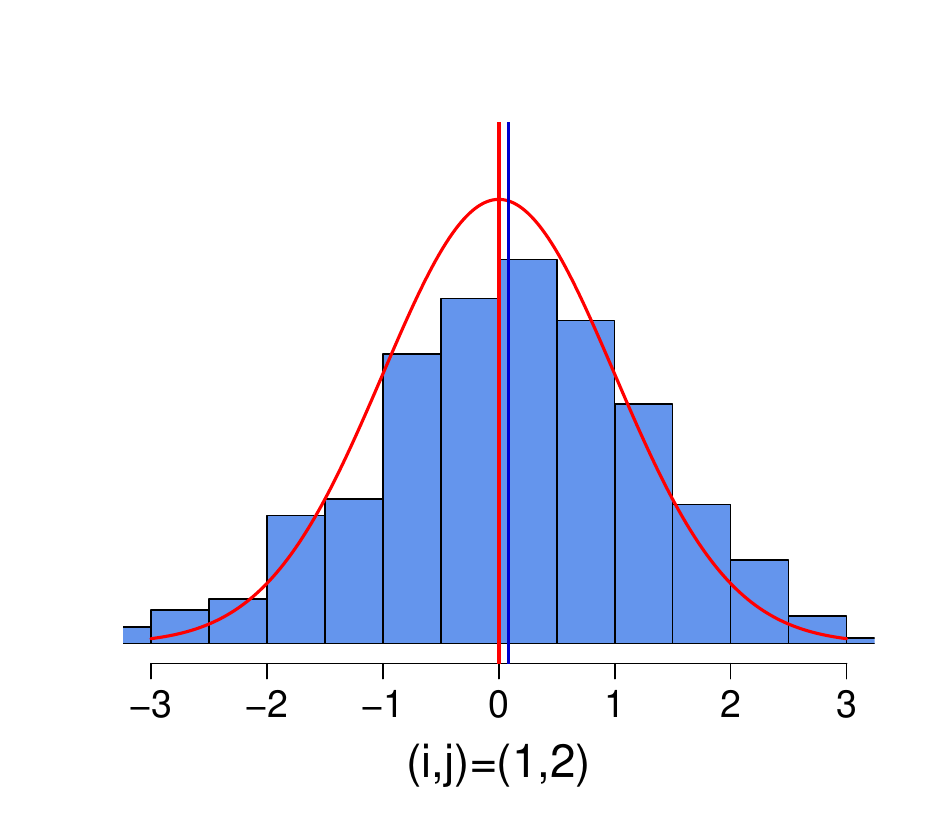}
    \end{minipage}
    \begin{minipage}{0.24\linewidth}
        \centering
        \includegraphics[width=\textwidth]{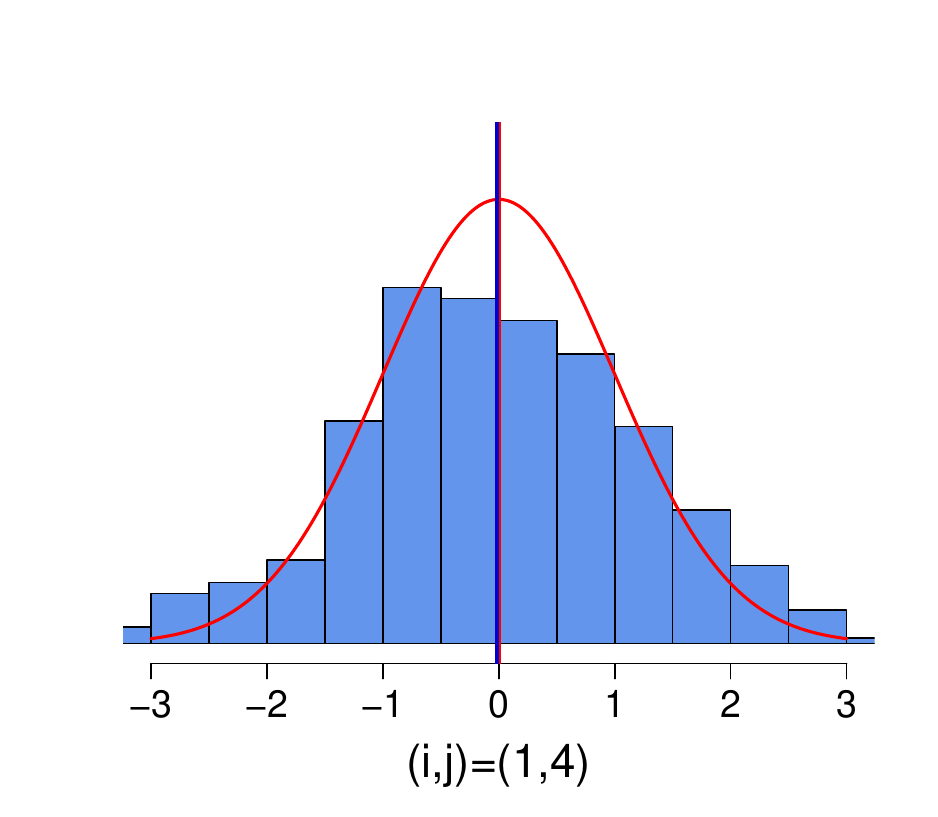}
    \end{minipage}
    \begin{minipage}{0.24\linewidth}
        \centering
        \includegraphics[width=\textwidth]{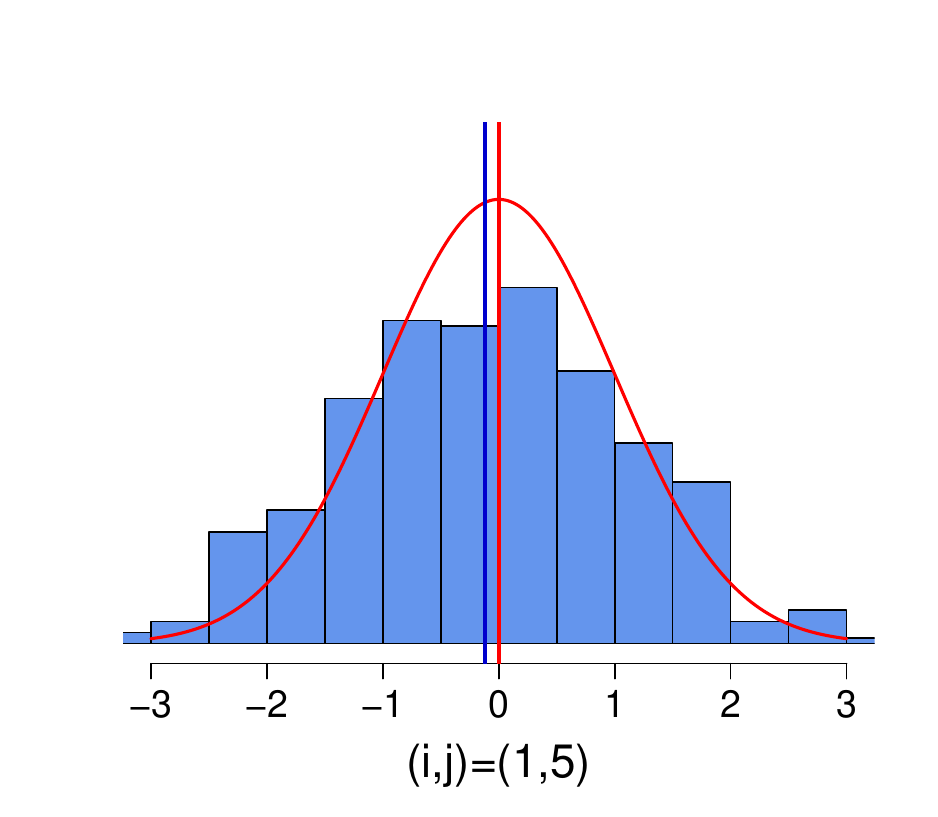}
    \end{minipage}
    \begin{minipage}{0.24\linewidth}
        \centering
        \includegraphics[width=\textwidth]{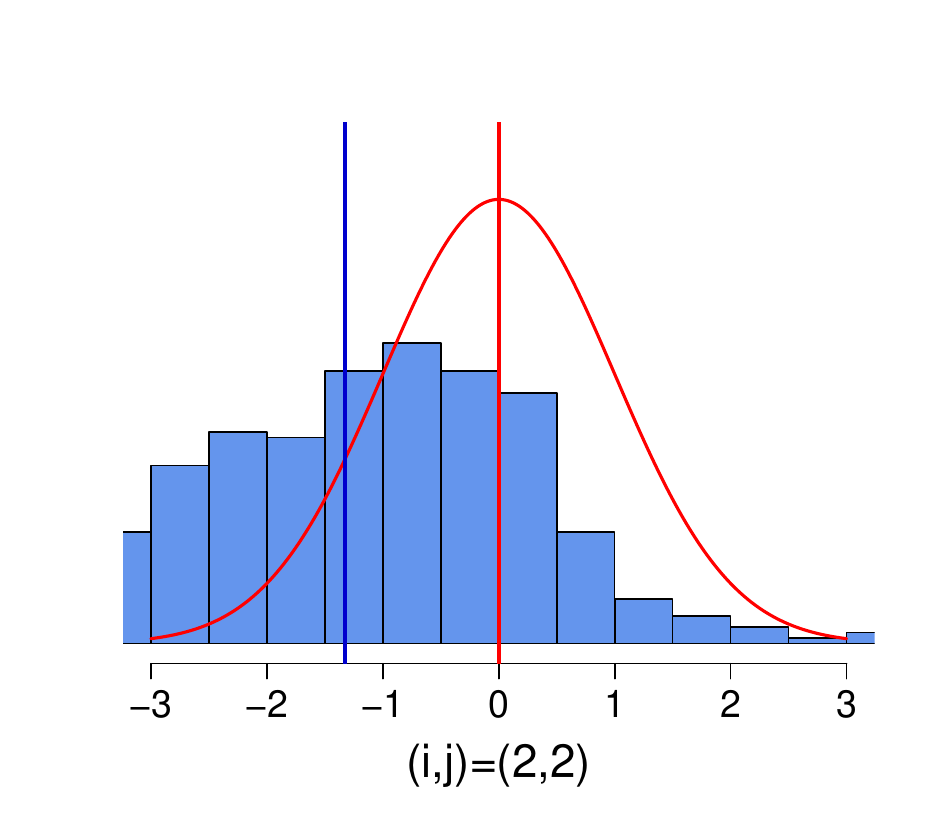}
    \end{minipage}
 \end{minipage}

  \caption*{$n=800, p=400$}
      \vspace{-0.43cm}
 \begin{minipage}{0.3\linewidth}
    \begin{minipage}{0.24\linewidth}
        \centering
        \includegraphics[width=\textwidth]{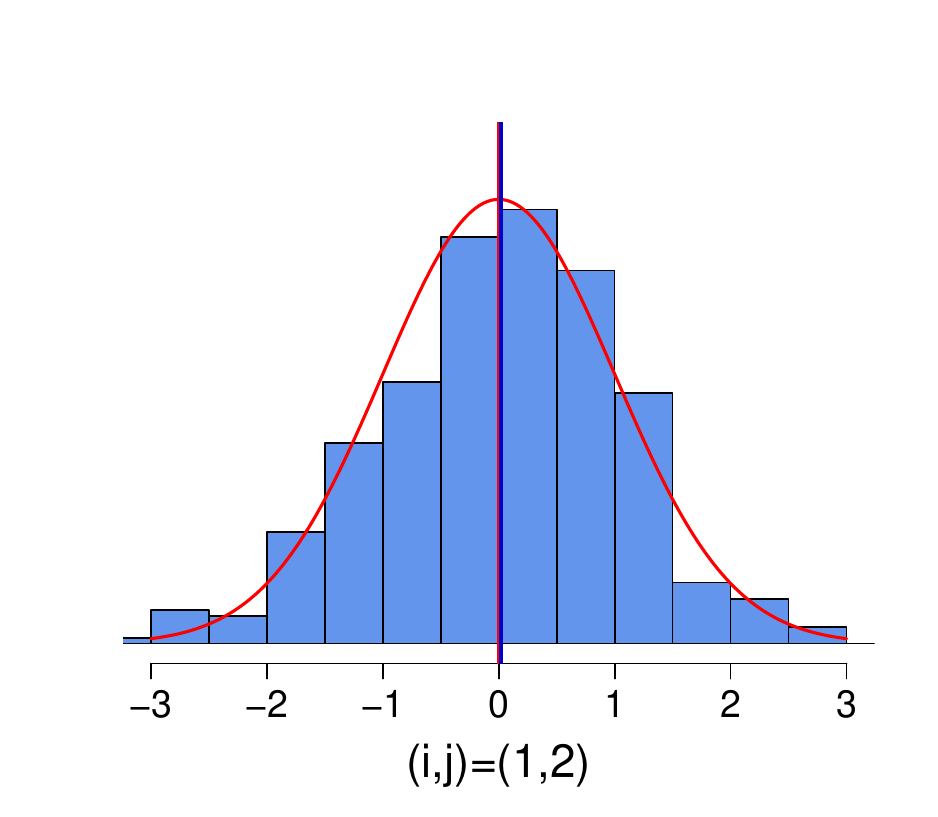}
    \end{minipage}
    \begin{minipage}{0.24\linewidth}
        \centering
        \includegraphics[width=\textwidth]{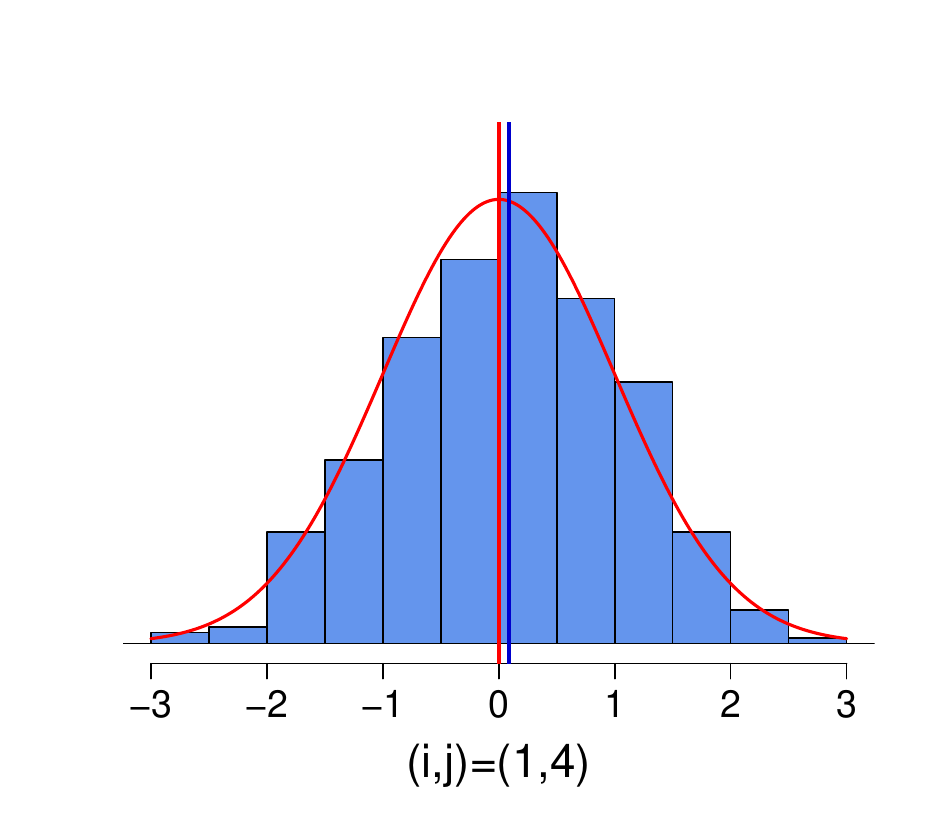}
    \end{minipage}
    \begin{minipage}{0.24\linewidth}
        \centering
        \includegraphics[width=\textwidth]{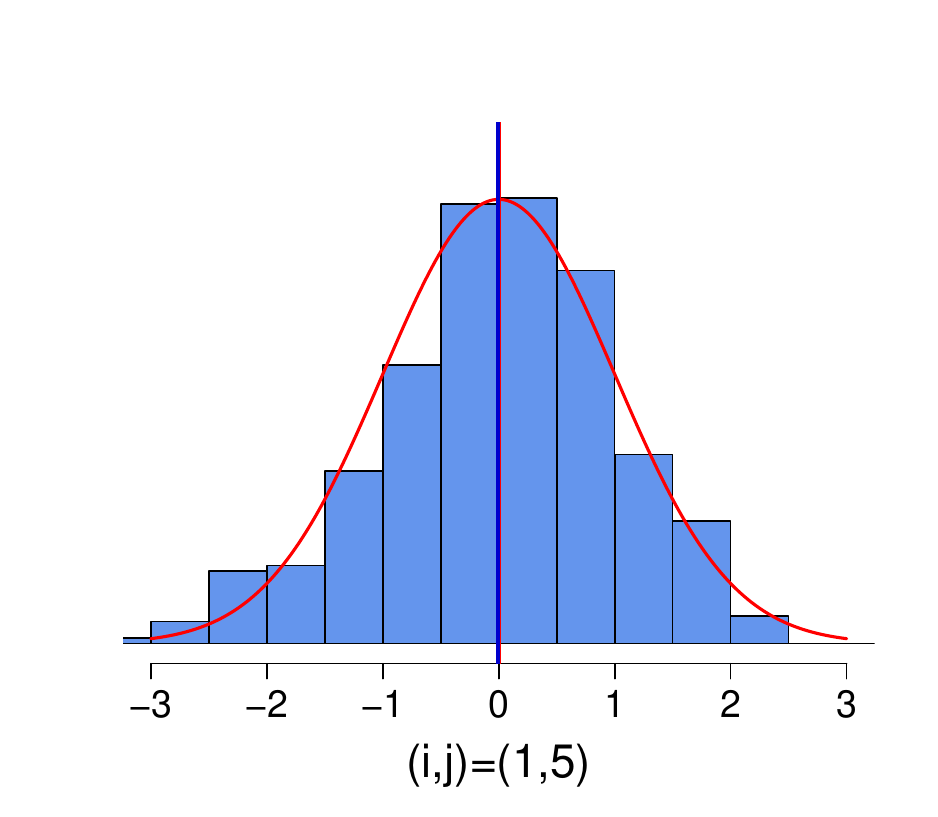}
    \end{minipage}
    \begin{minipage}{0.24\linewidth}
        \centering
        \includegraphics[width=\textwidth]{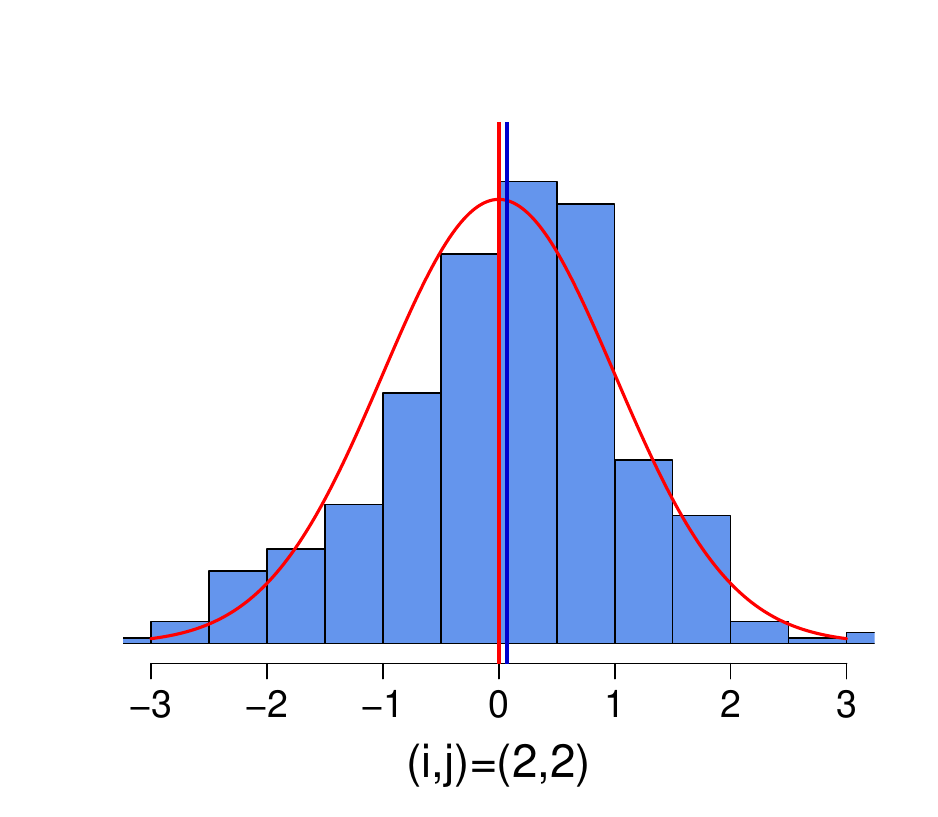}
    \end{minipage}
   \caption*{(a)~~$L_0{:}~ \widehat{\mb{\Omega}}^{\text{US}}$}
 \end{minipage} 
     \hspace{1cm}
 \begin{minipage}{0.3\linewidth}
    \begin{minipage}{0.24\linewidth}
        \centering
        \includegraphics[width=\textwidth]{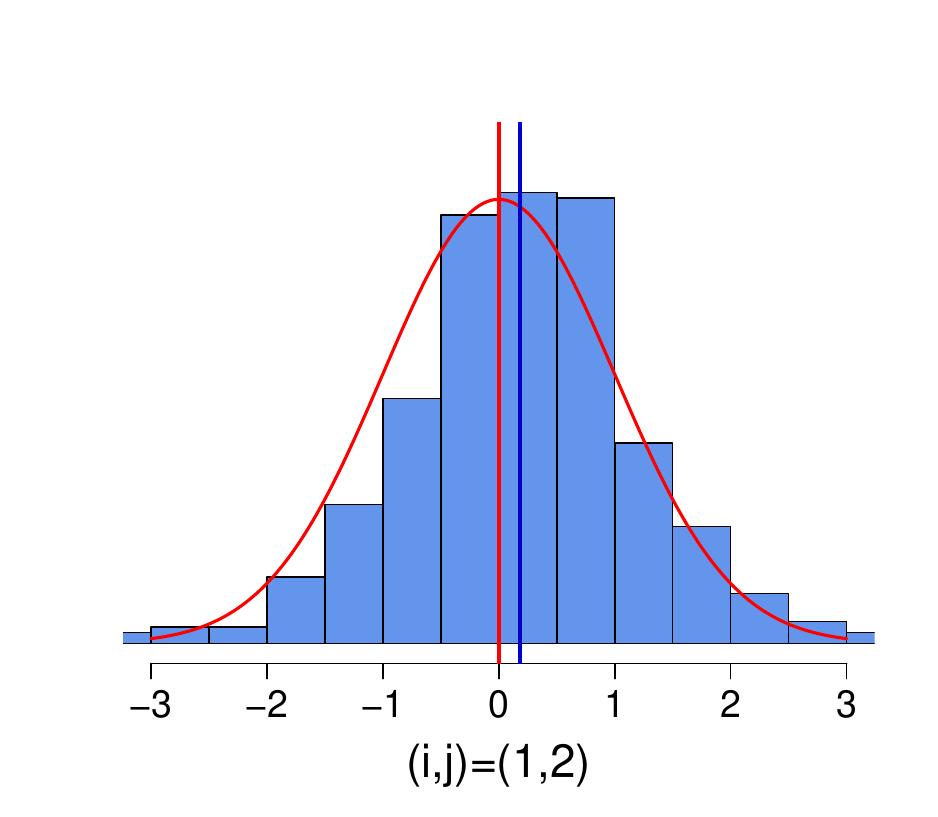}
    \end{minipage}
    \begin{minipage}{0.24\linewidth}
        \centering
        \includegraphics[width=\textwidth]{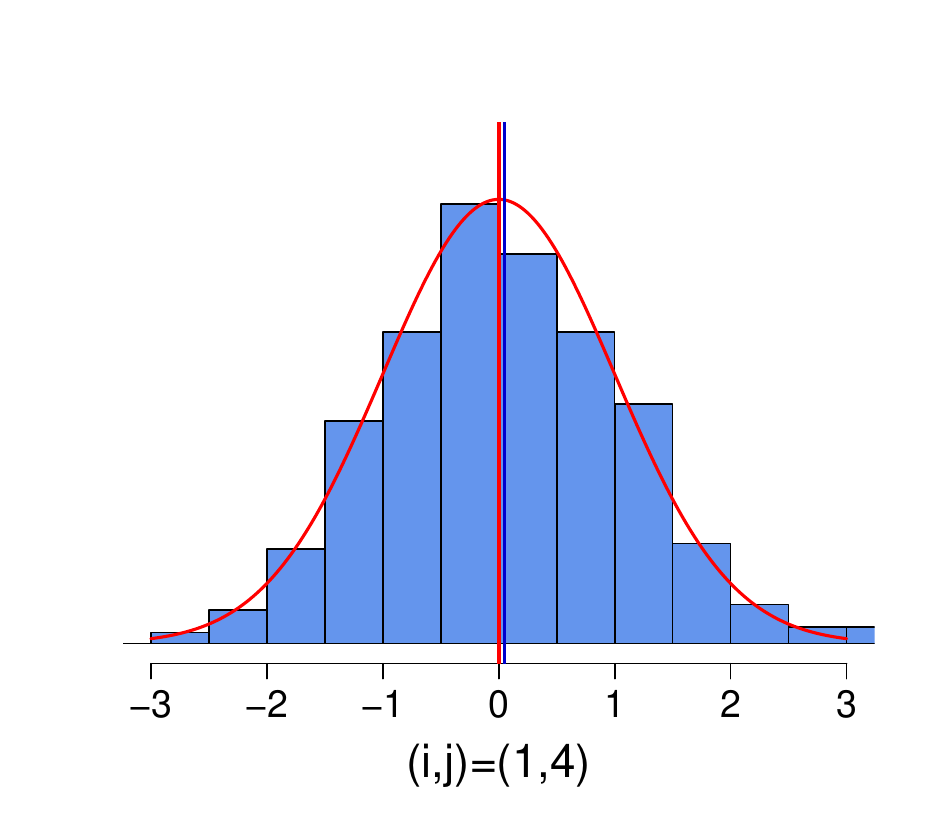}
    \end{minipage}
    \begin{minipage}{0.24\linewidth}
        \centering
        \includegraphics[width=\textwidth]{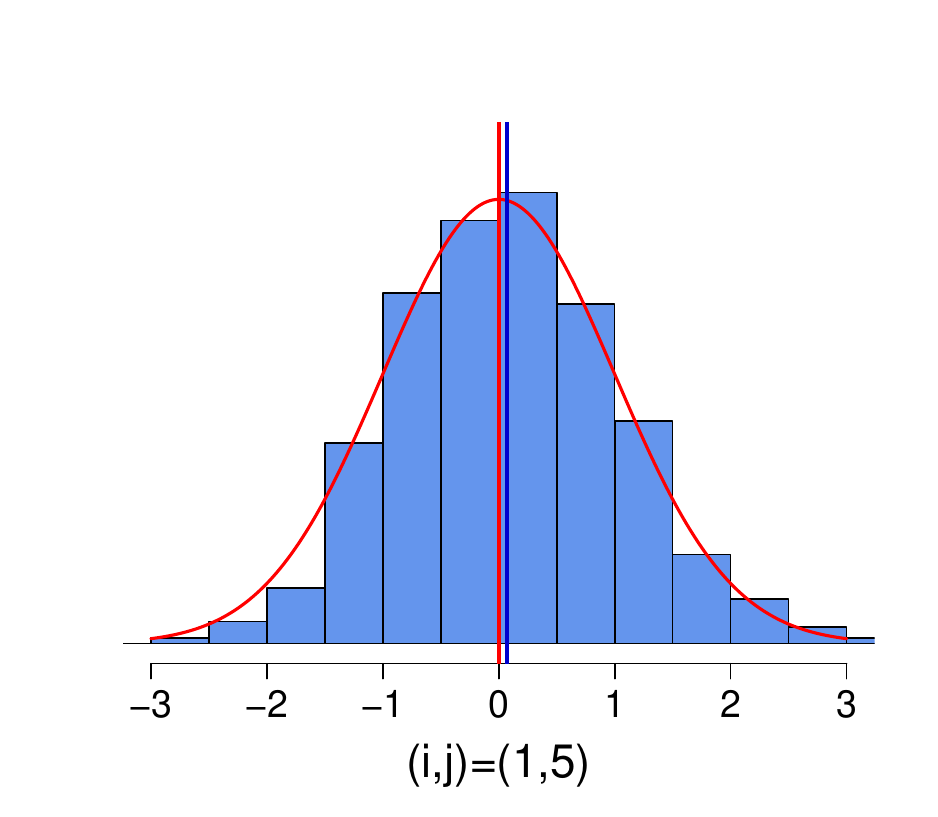}
    \end{minipage}
    \begin{minipage}{0.24\linewidth}
        \centering
        \includegraphics[width=\textwidth]{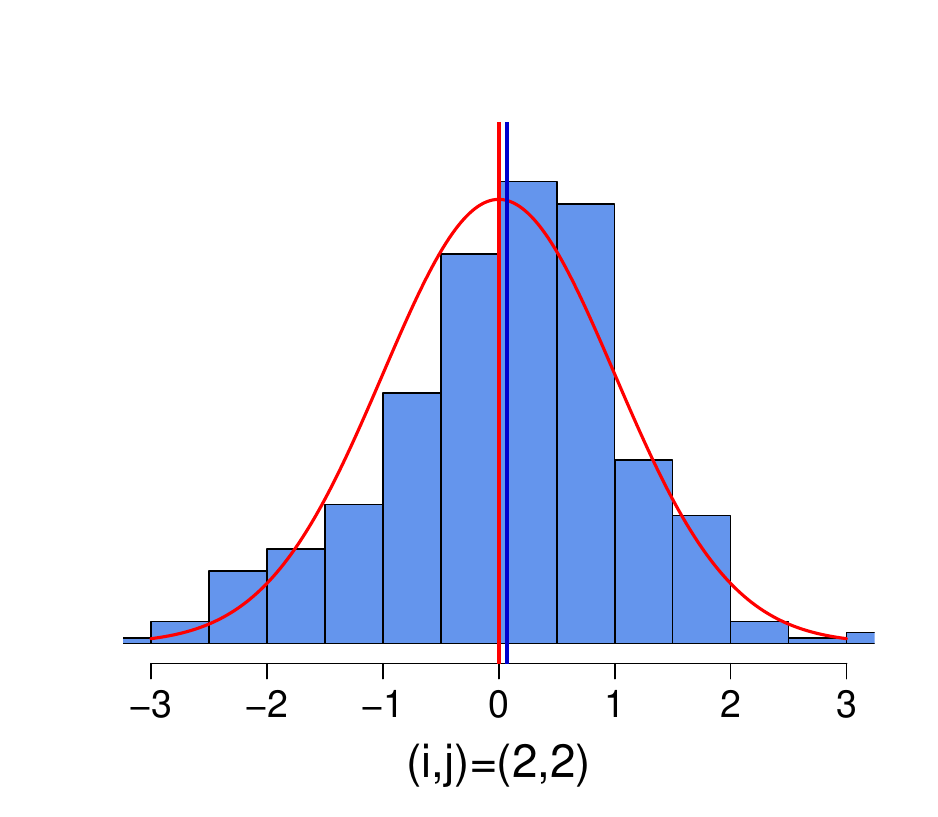}
    \end{minipage}
        \caption*{(b)~~$L_0{:}~ \widehat{\mb{T}}$}
 \end{minipage}   
      \hspace{1cm}
 \begin{minipage}{0.3\linewidth}
    \begin{minipage}{0.24\linewidth}
        \centering
        \includegraphics[width=\textwidth]{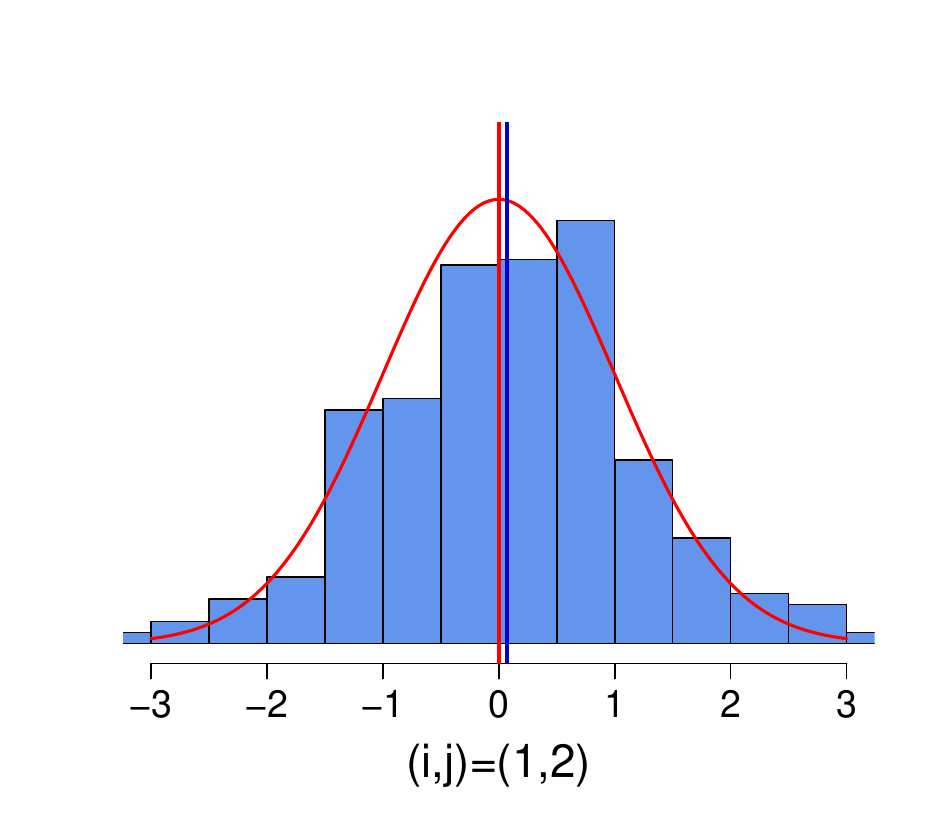}
    \end{minipage}
    \begin{minipage}{0.24\linewidth}
        \centering
        \includegraphics[width=\textwidth]{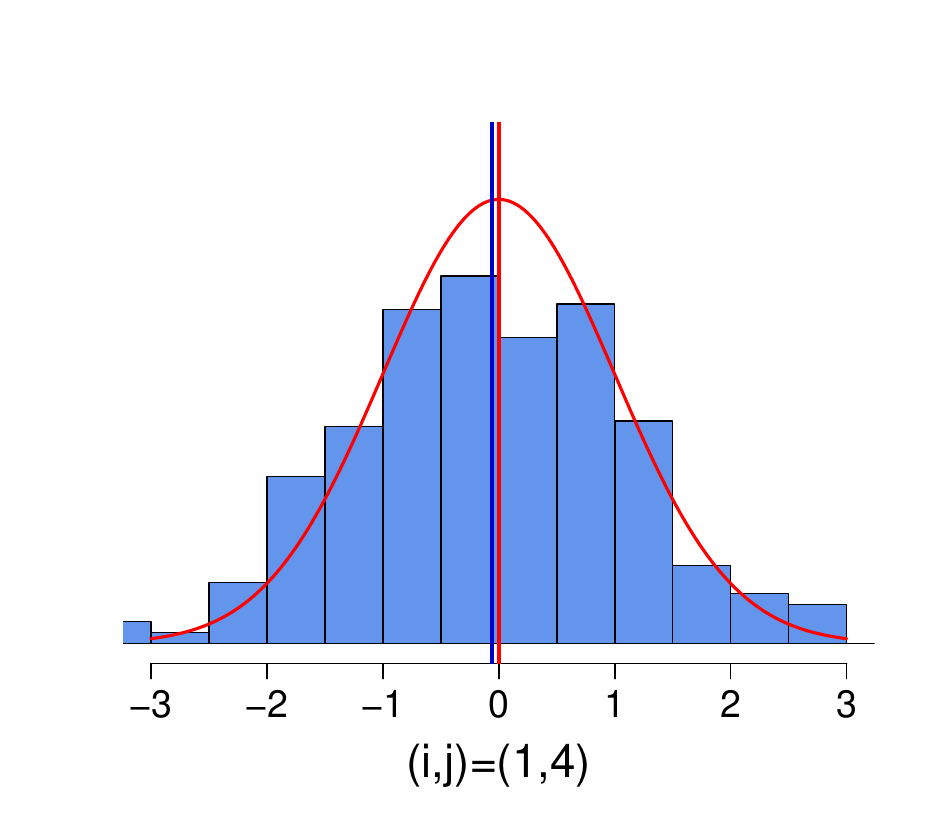}
    \end{minipage}
    \begin{minipage}{0.24\linewidth}
        \centering
        \includegraphics[width=\textwidth]{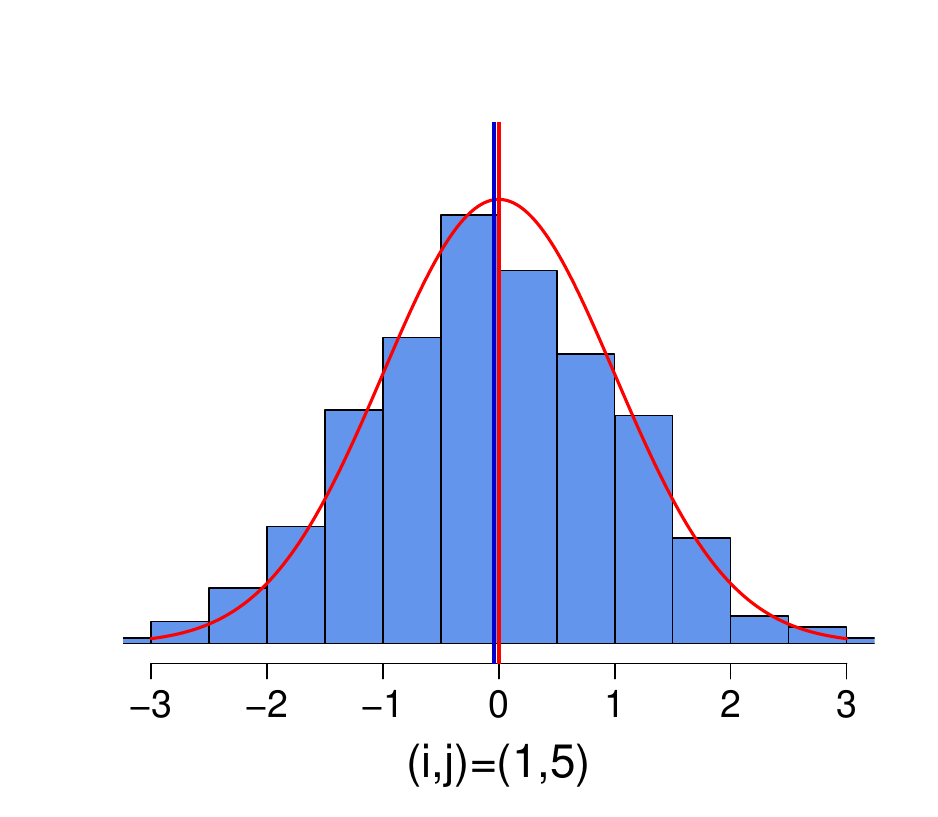}
    \end{minipage}
    \begin{minipage}{0.24\linewidth}
        \centering
        \includegraphics[width=\textwidth]{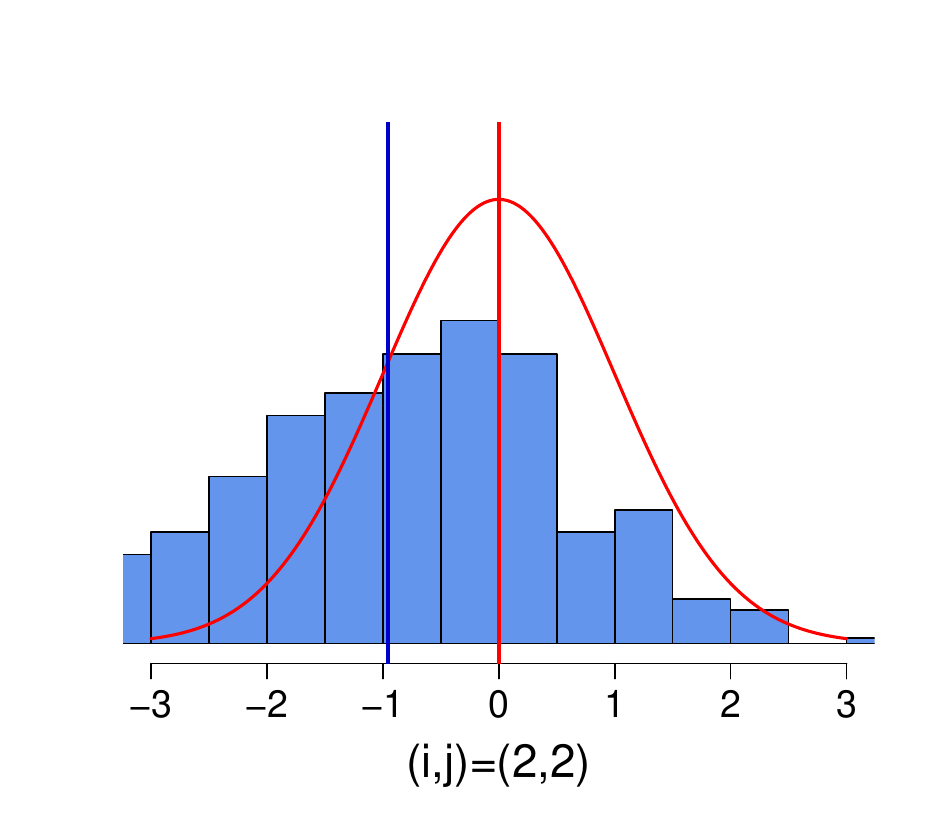}
    \end{minipage}
    \caption*{(c)~~$L_1{:}~ \widehat{\mb{T}}$}
     \end{minipage}   
     \caption{Histograms of $\big(\sqrt{n}(\widehat{\mb{\Omega}}_{ij}^{(m)}-\mb{\Omega}_{ij})/\widehat{\sigma}_{\mb{\Omega}_{ij}}^{(m)}\big)_{m=1}^{400}$ under sub-Gaussian hub graph settings.}
\end{sidewaysfigure}

%sub-Gaussian cluster
 \begin{sidewaysfigure}[th!]
  \caption*{$n=200, p=200$}
      \vspace{-0.43cm}
 \begin{minipage}{0.3\linewidth}
    \begin{minipage}{0.24\linewidth}
        \centering
        \includegraphics[width=\textwidth]{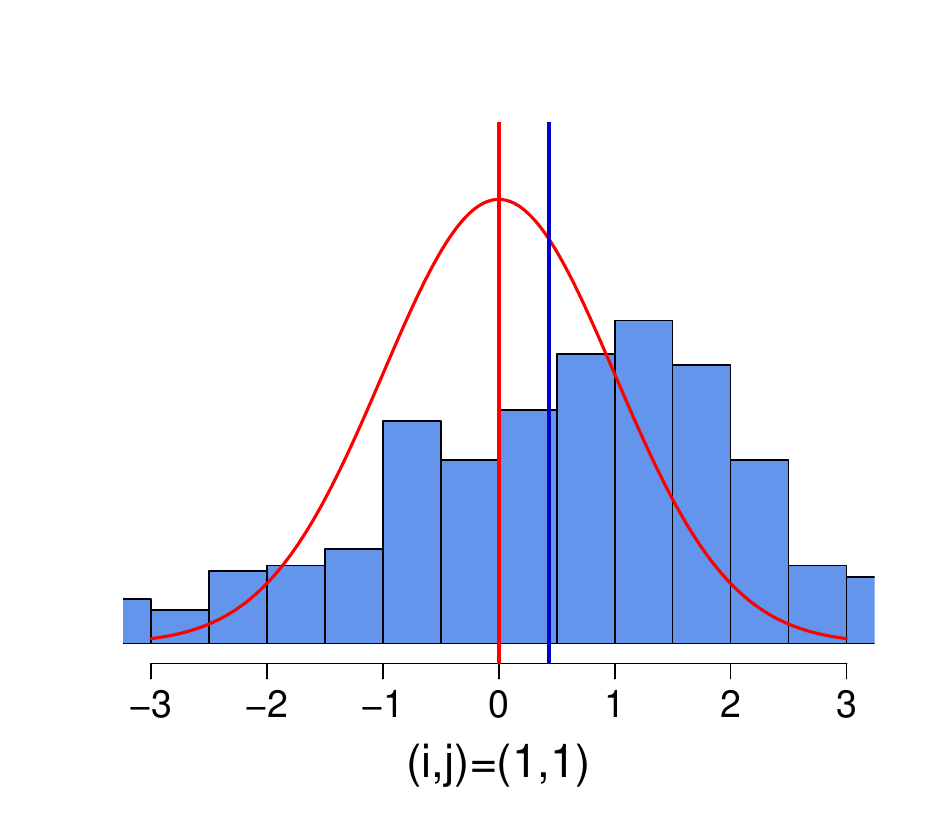}
    \end{minipage}
    \begin{minipage}{0.24\linewidth}
        \centering
        \includegraphics[width=\textwidth]{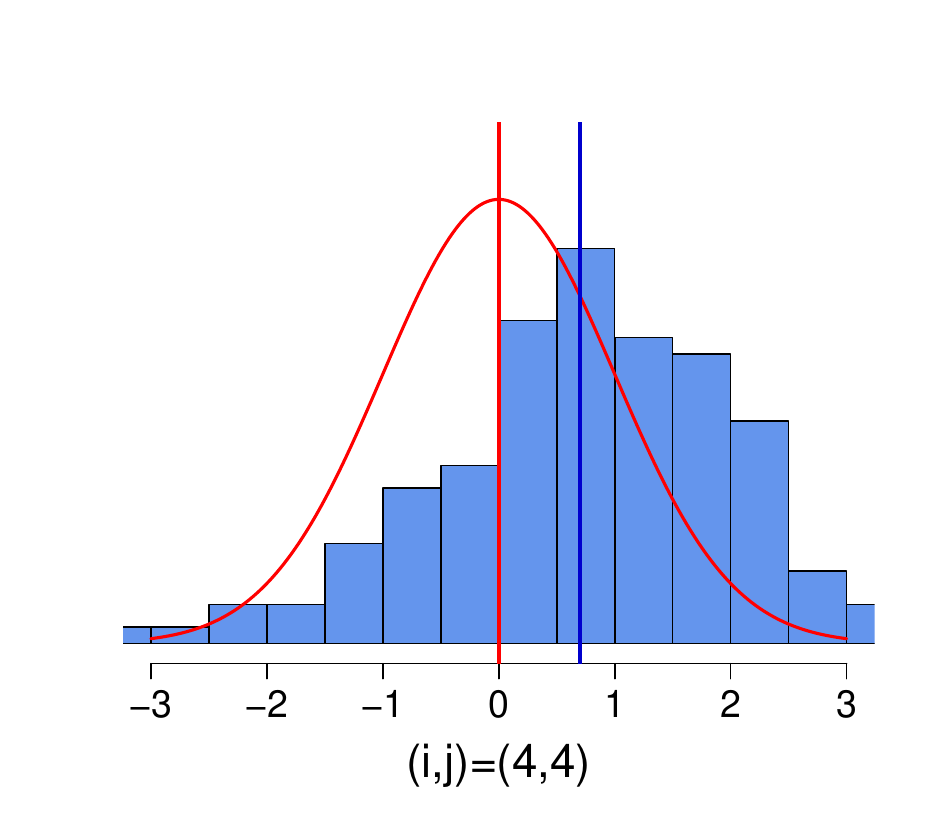}
    \end{minipage}
    \begin{minipage}{0.24\linewidth}
        \centering
        \includegraphics[width=\textwidth]{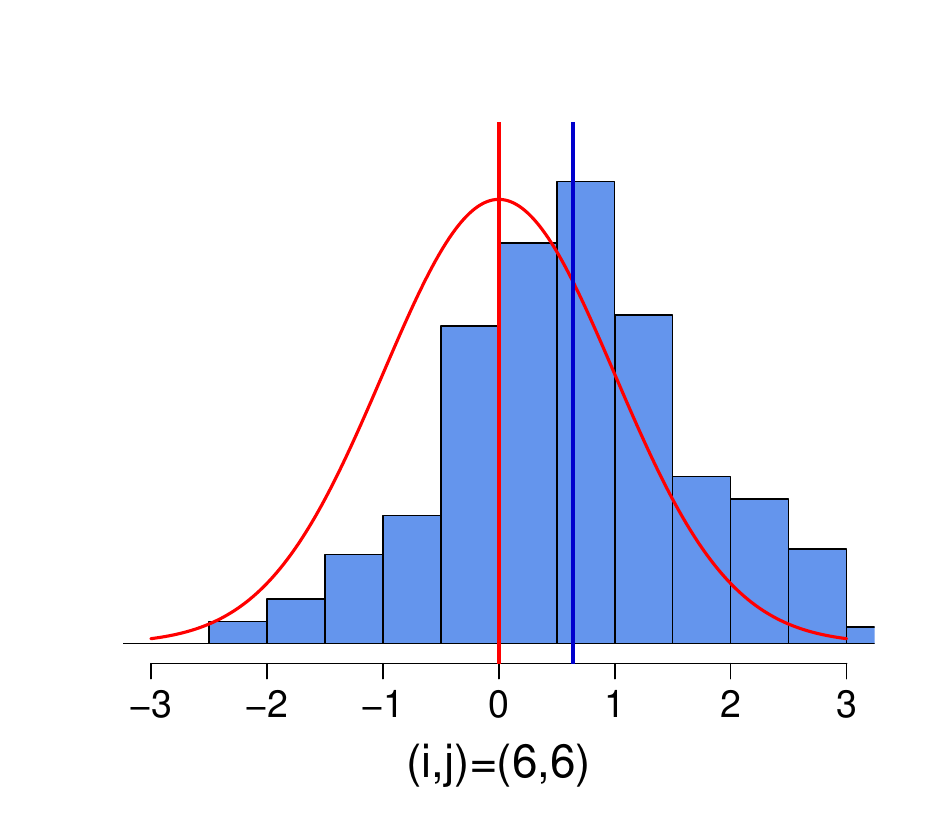}
    \end{minipage}
    \begin{minipage}{0.24\linewidth}
        \centering
        \includegraphics[width=\textwidth]{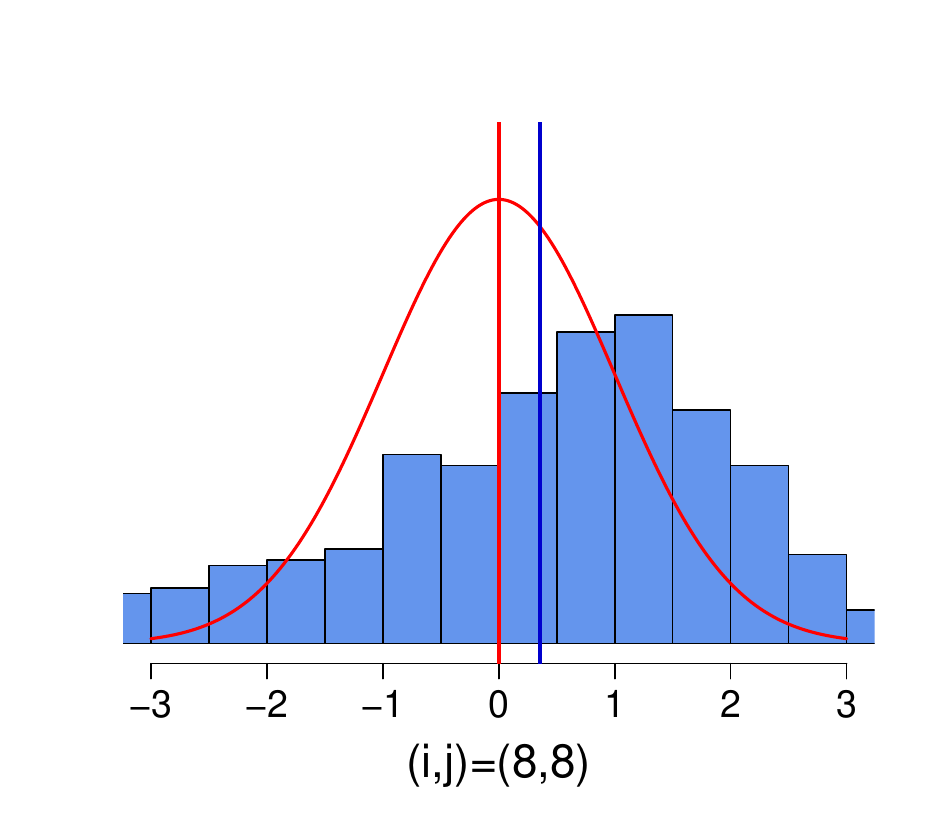}
    \end{minipage}
 \end{minipage}
 \hspace{1cm}
 \begin{minipage}{0.3\linewidth}
    \begin{minipage}{0.24\linewidth}
        \centering
        \includegraphics[width=\textwidth]{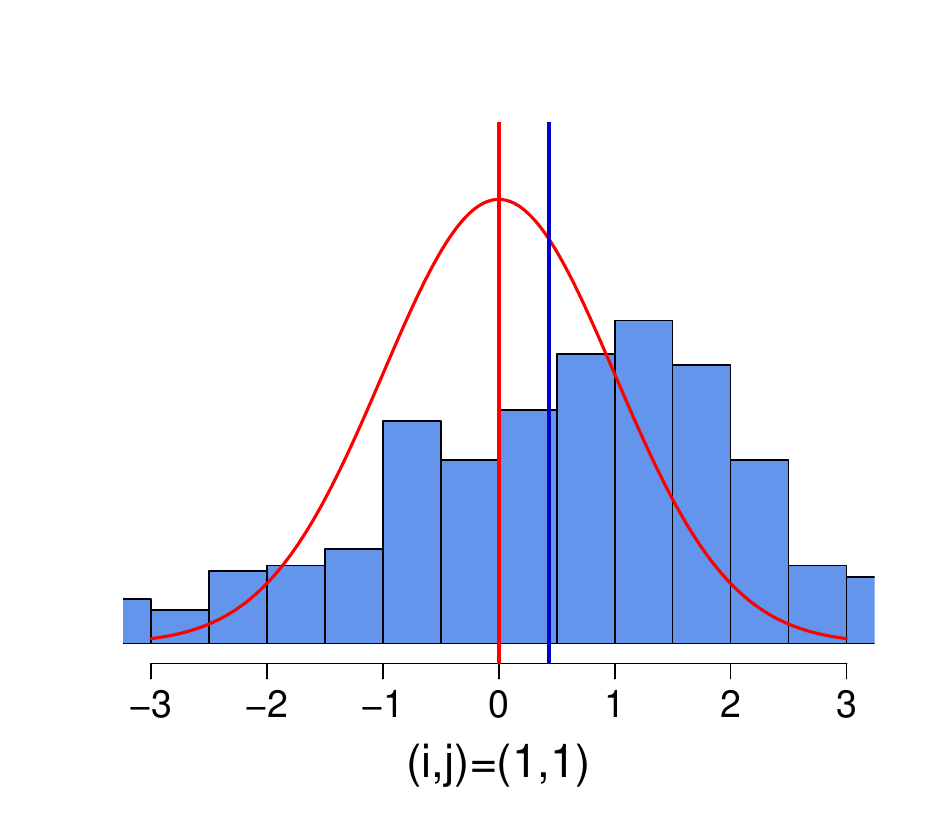}
    \end{minipage}
    \begin{minipage}{0.24\linewidth}
        \centering
        \includegraphics[width=\textwidth]{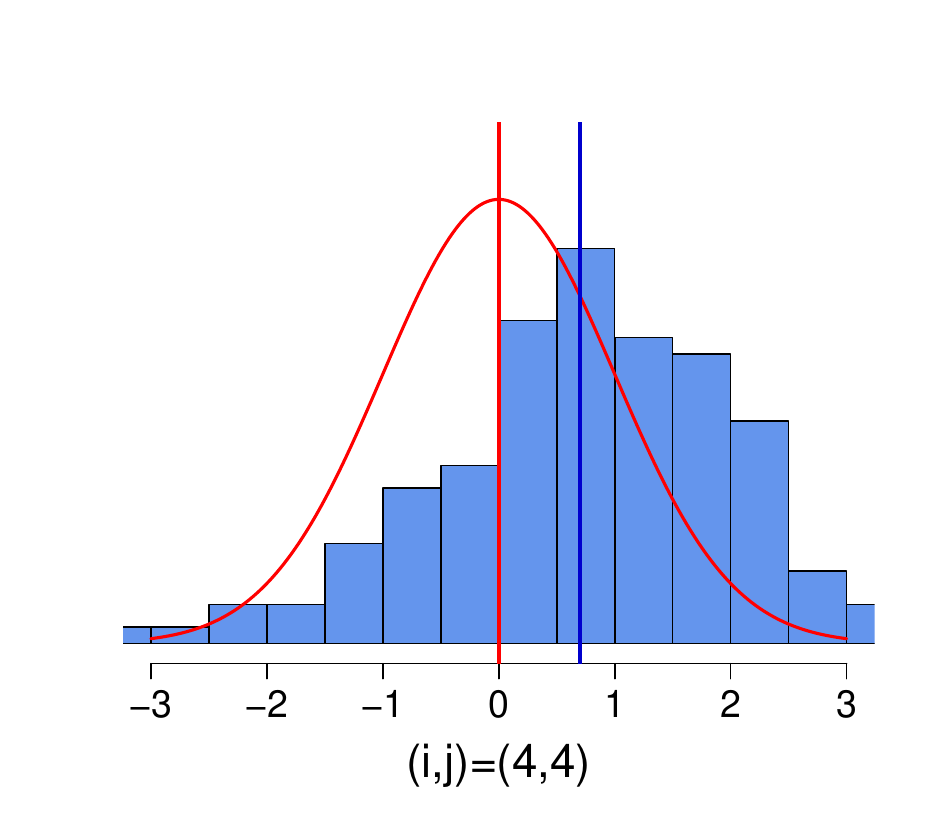}
    \end{minipage}
    \begin{minipage}{0.24\linewidth}
        \centering
        \includegraphics[width=\textwidth]{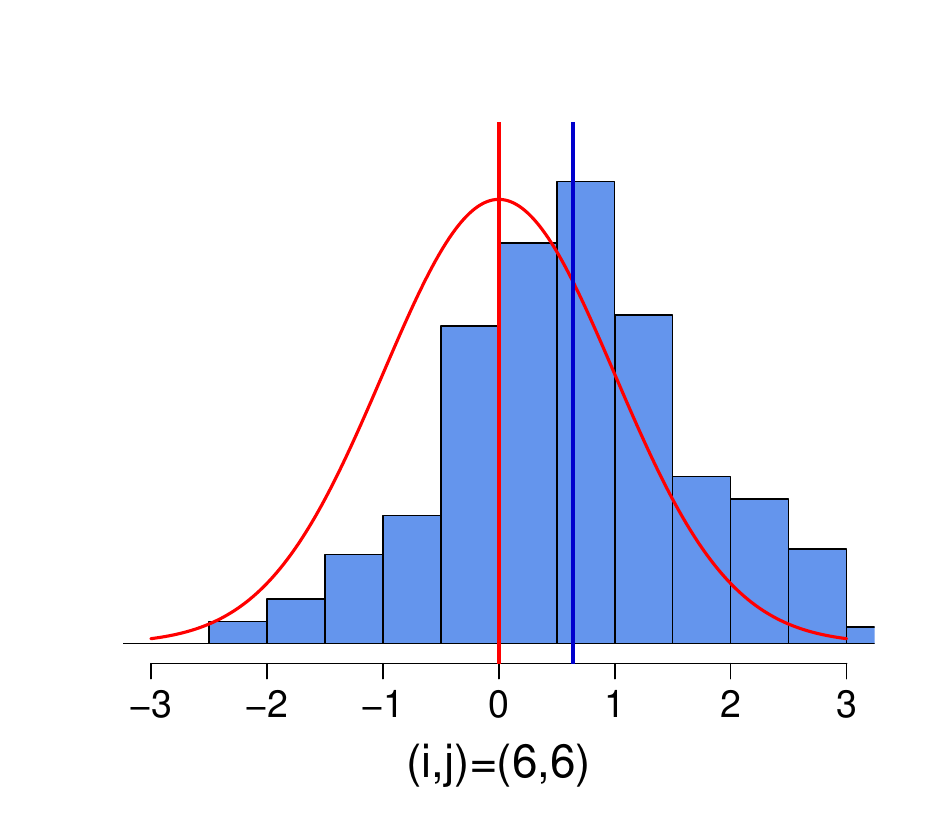}
    \end{minipage}
    \begin{minipage}{0.24\linewidth}
        \centering
        \includegraphics[width=\textwidth]{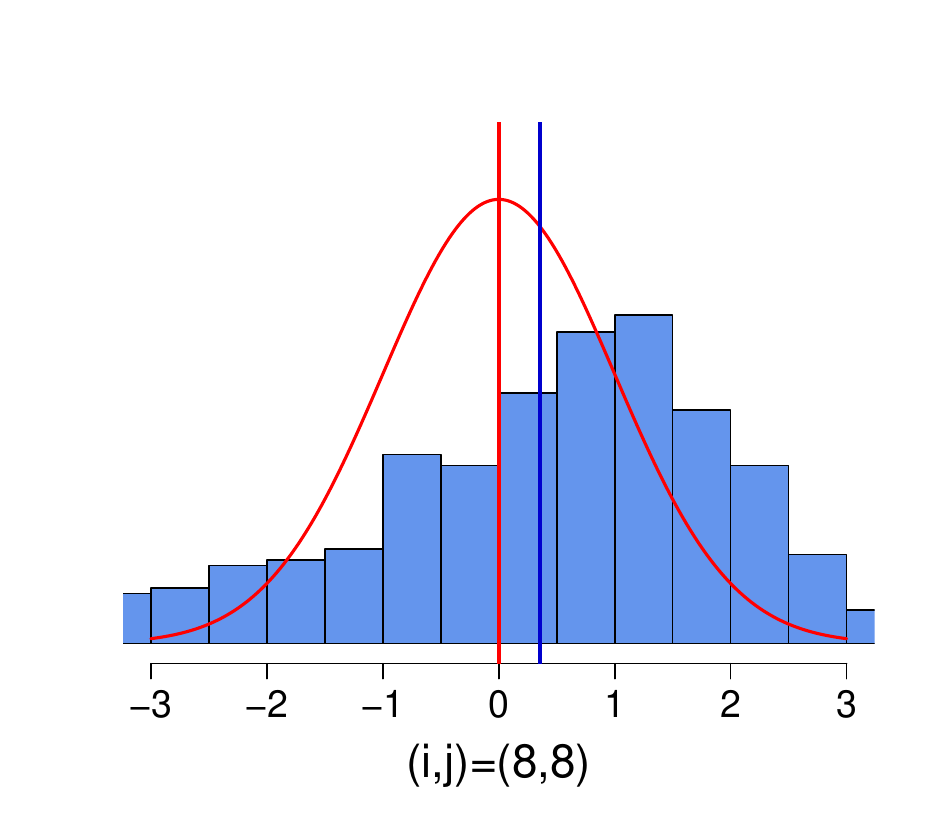}
    \end{minipage}    
 \end{minipage}
  \hspace{1cm}
 \begin{minipage}{0.3\linewidth}
     \begin{minipage}{0.24\linewidth}
        \centering
        \includegraphics[width=\textwidth]{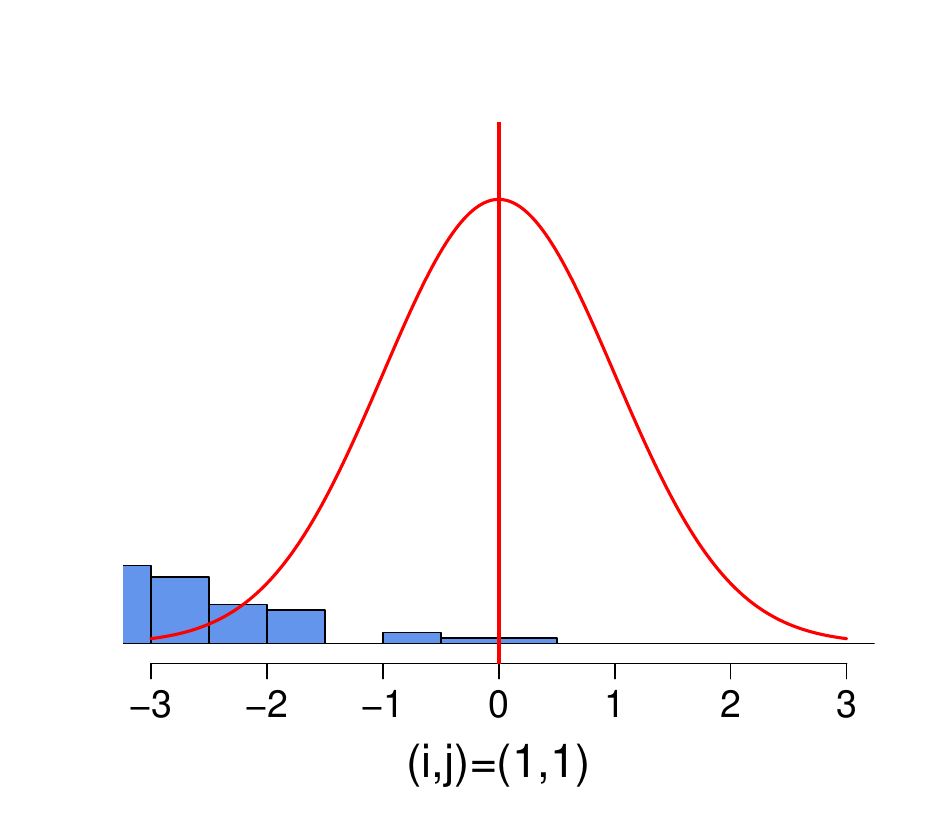}
    \end{minipage}
    \begin{minipage}{0.24\linewidth}
        \centering
        \includegraphics[width=\textwidth]{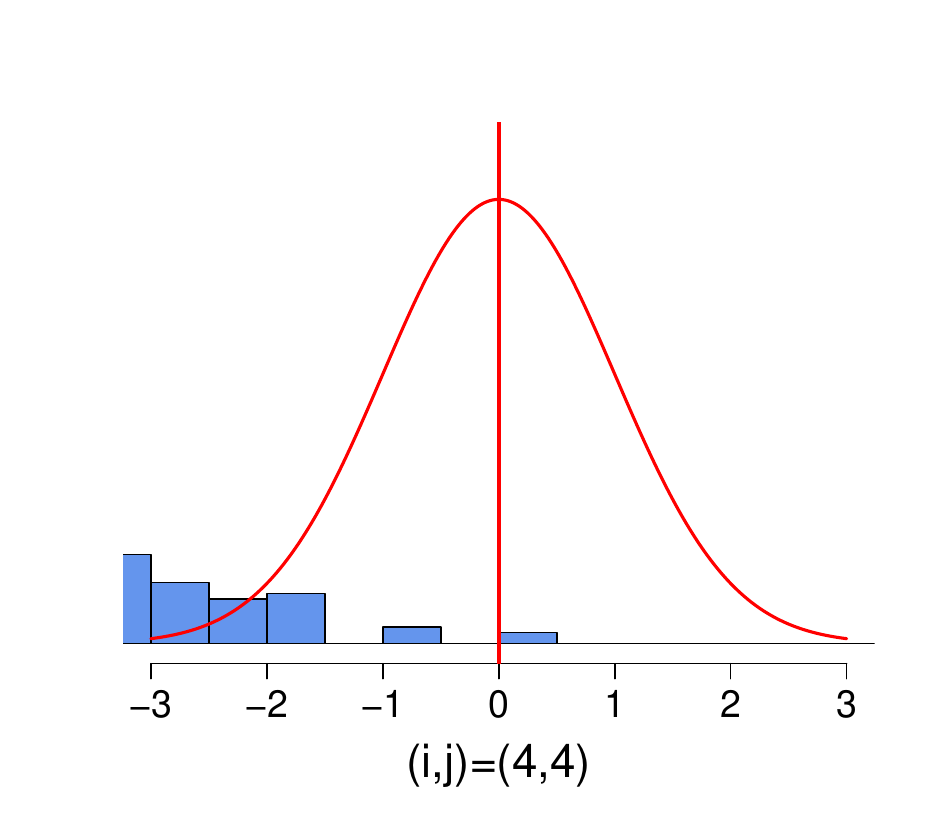}
    \end{minipage}
    \begin{minipage}{0.24\linewidth}
        \centering
        \includegraphics[width=\textwidth]{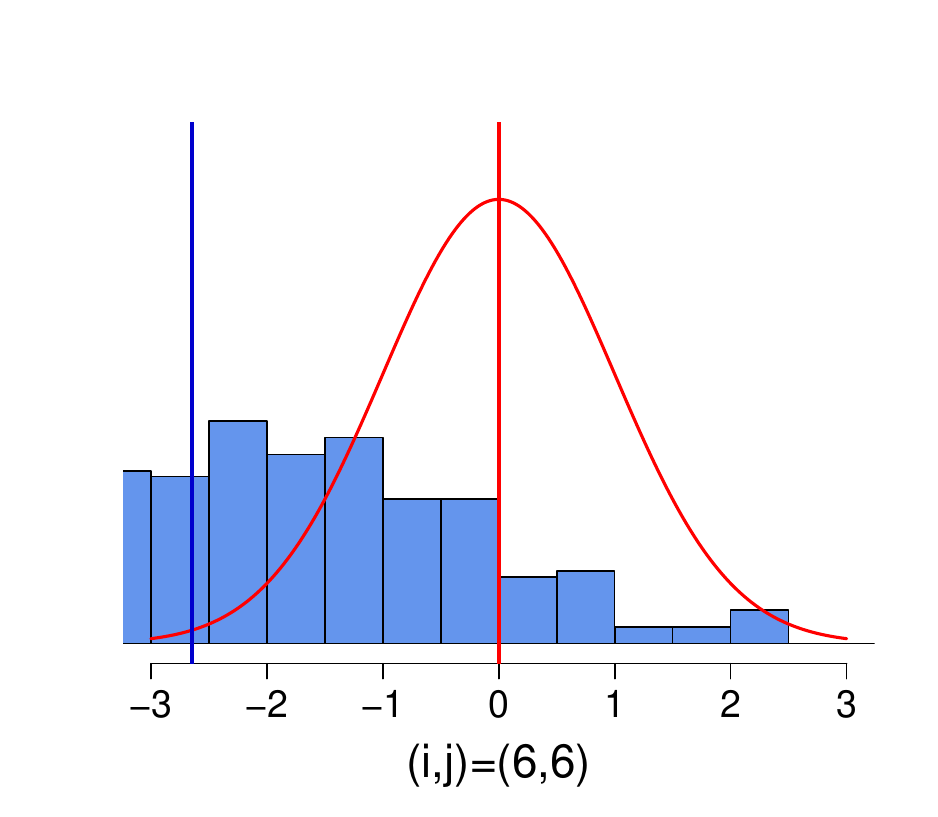}
    \end{minipage}
    \begin{minipage}{0.24\linewidth}
        \centering
        \includegraphics[width=\textwidth]{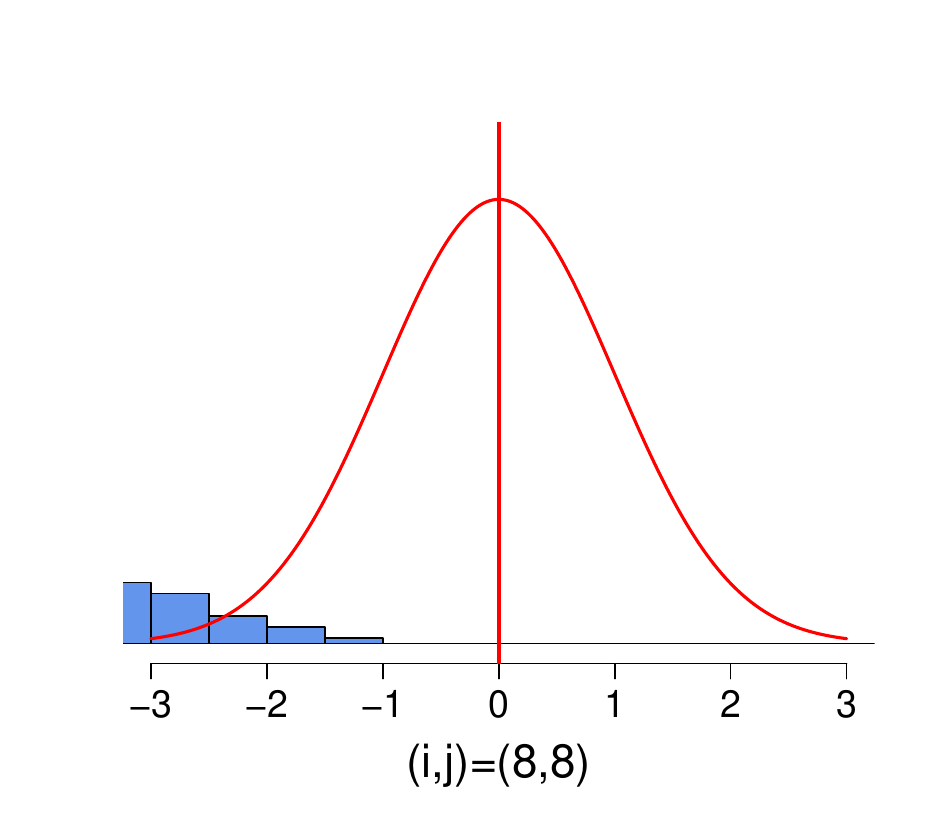}
    \end{minipage}
 \end{minipage}

  \caption*{$n=400, p=200$}
      \vspace{-0.43cm}
 \begin{minipage}{0.3\linewidth}
    \begin{minipage}{0.24\linewidth}
        \centering
        \includegraphics[width=\textwidth]{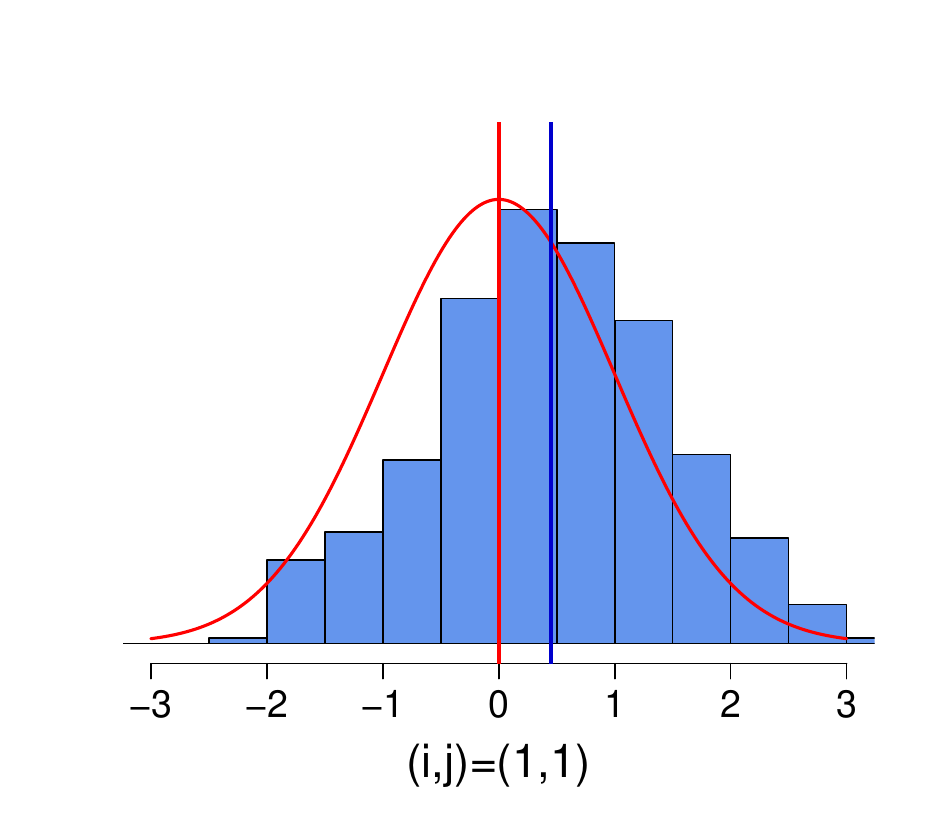}
    \end{minipage}
    \begin{minipage}{0.24\linewidth}
        \centering
        \includegraphics[width=\textwidth]{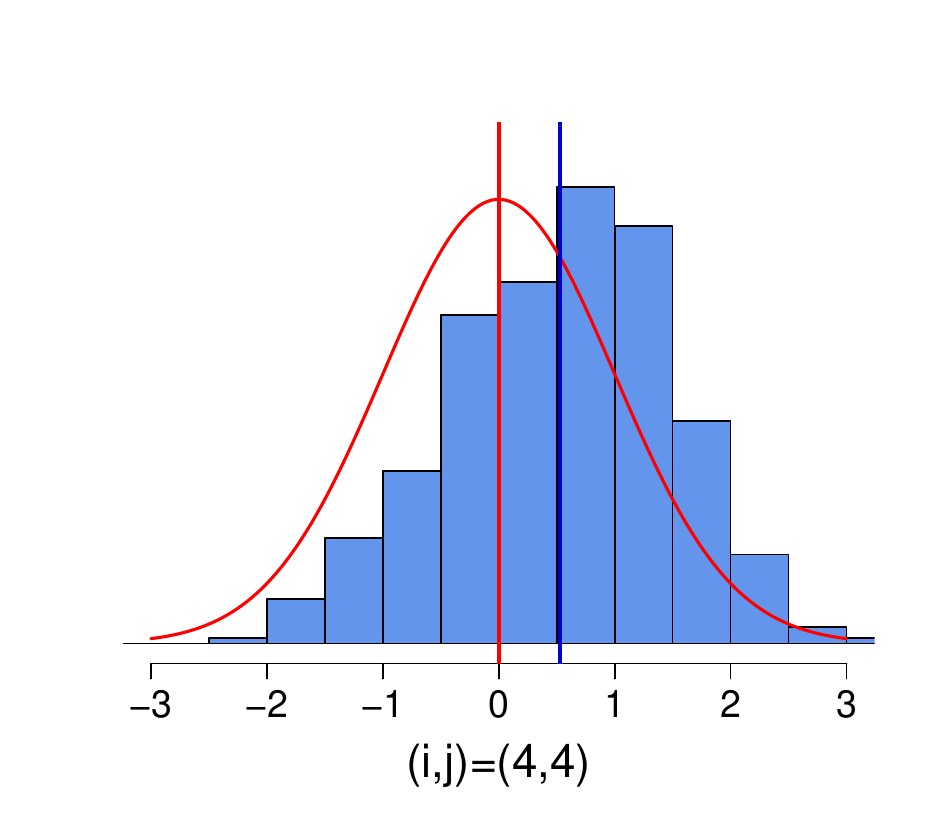}
    \end{minipage}
    \begin{minipage}{0.24\linewidth}
        \centering
        \includegraphics[width=\textwidth]{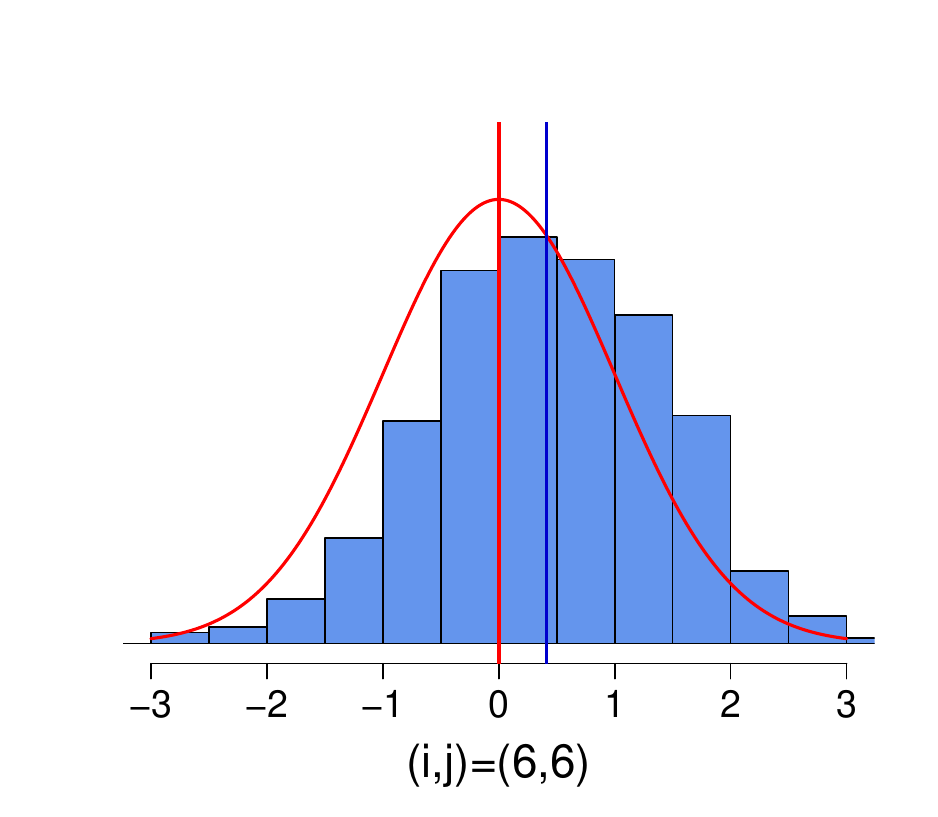}
    \end{minipage}
    \begin{minipage}{0.24\linewidth}
        \centering
        \includegraphics[width=\textwidth]{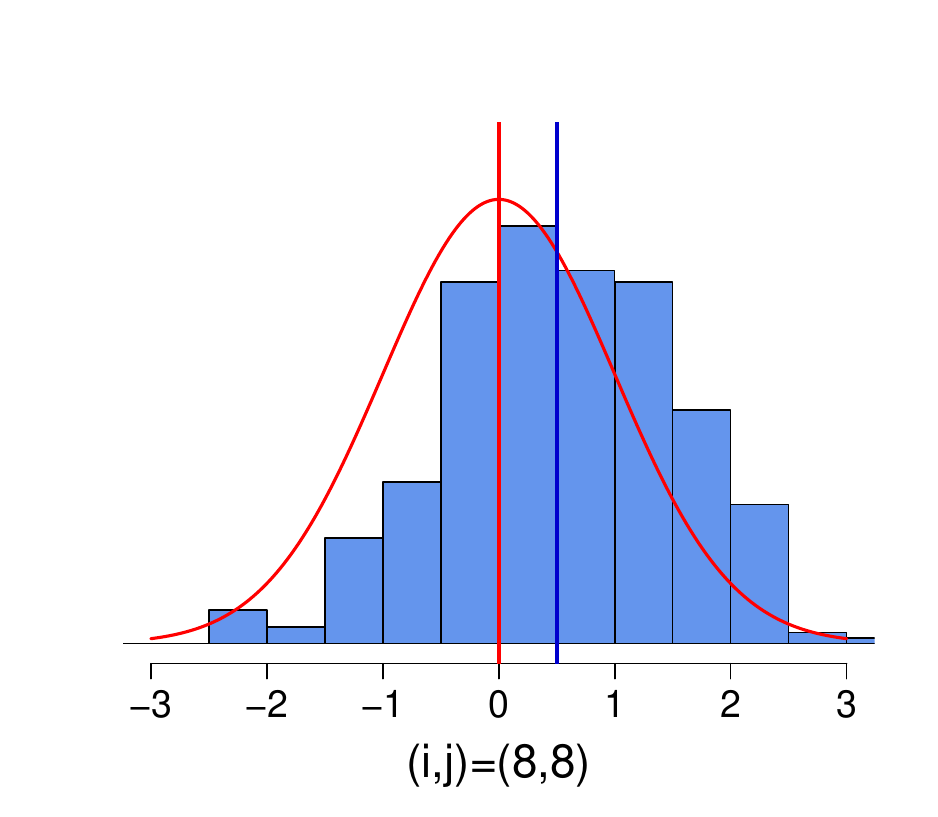}
    \end{minipage}
 \end{minipage}  
     \hspace{1cm}
 \begin{minipage}{0.3\linewidth}
    \begin{minipage}{0.24\linewidth}
        \centering
        \includegraphics[width=\textwidth]{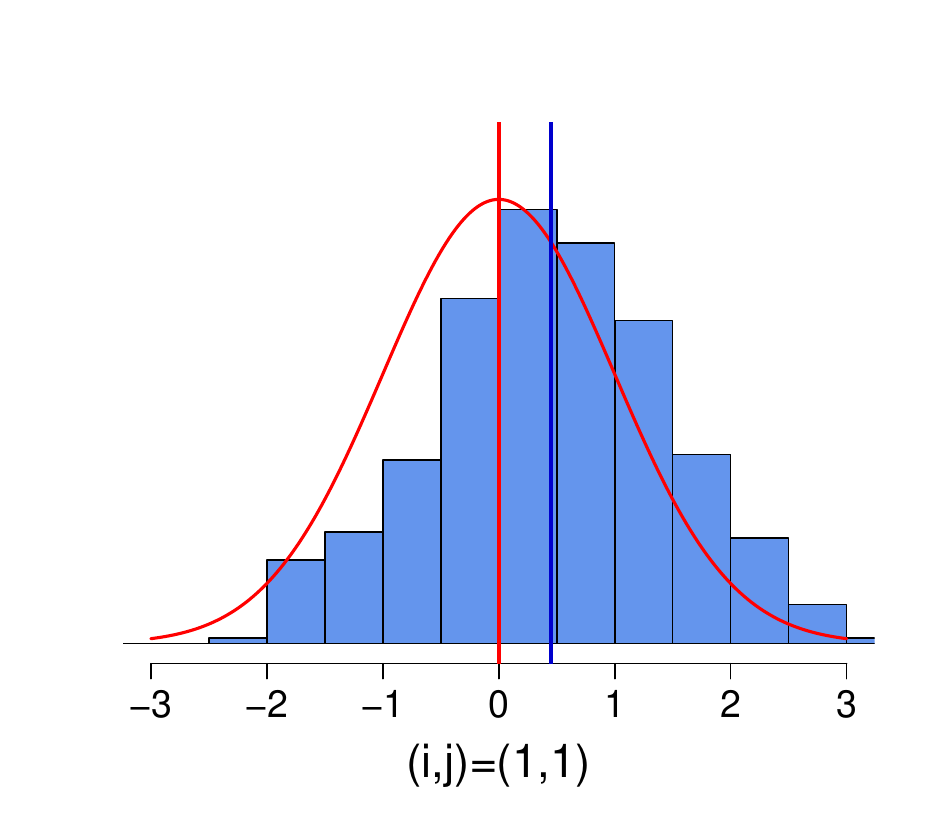}
    \end{minipage}
    \begin{minipage}{0.24\linewidth}
        \centering
        \includegraphics[width=\textwidth]{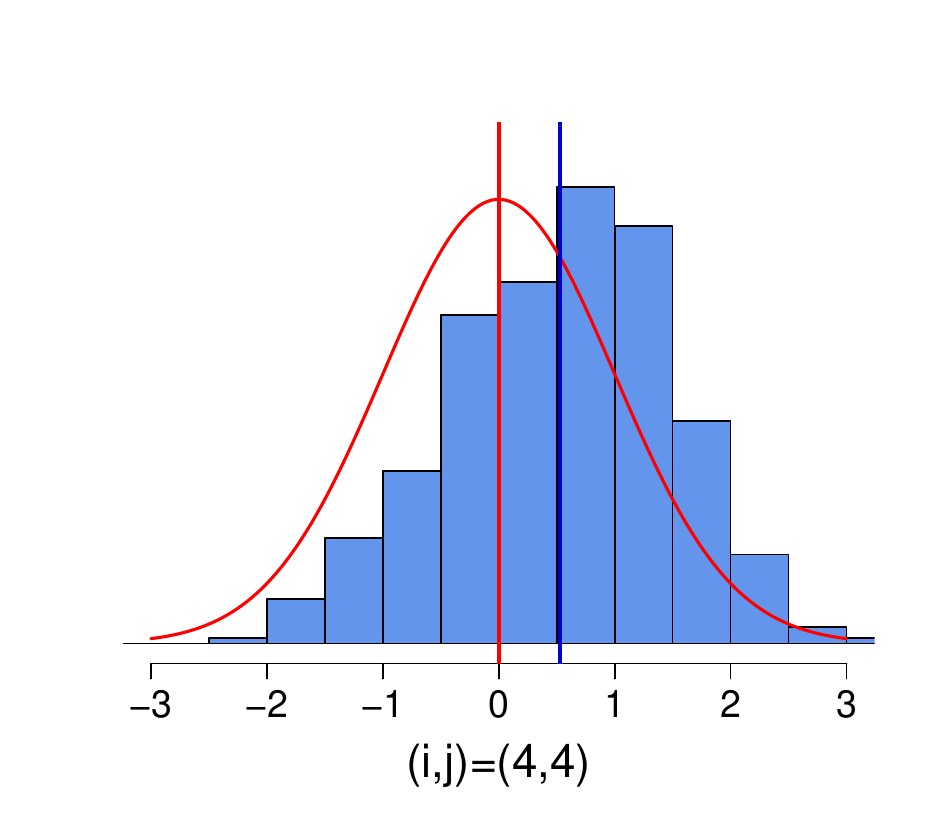}
    \end{minipage}
    \begin{minipage}{0.24\linewidth}
        \centering
        \includegraphics[width=\textwidth]{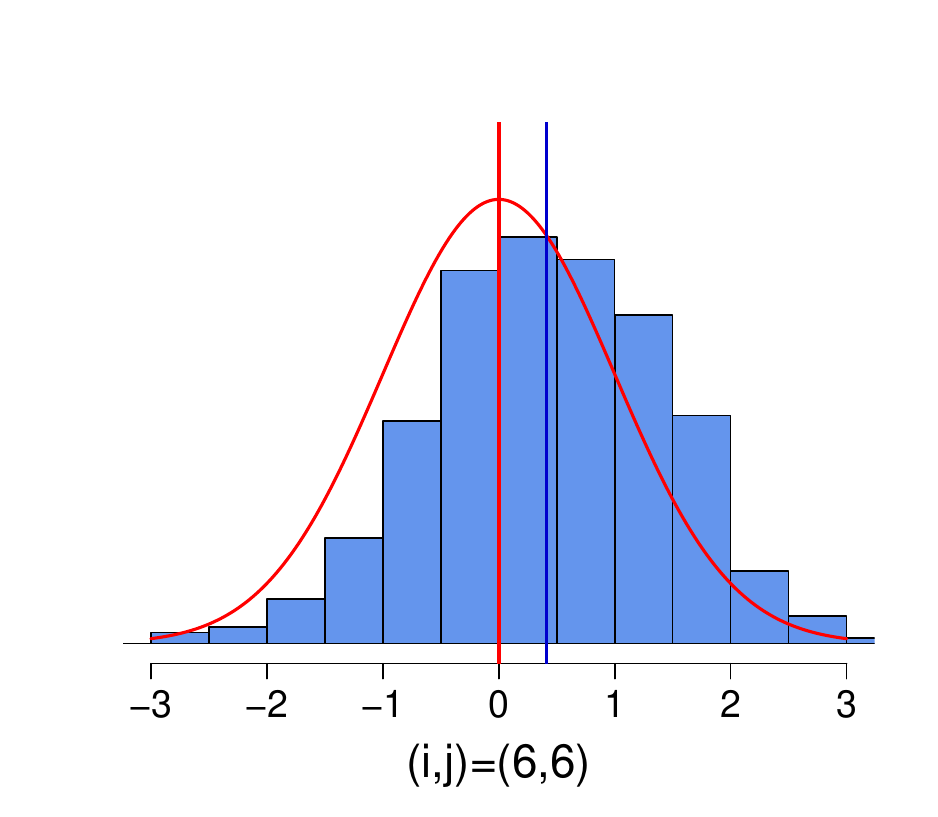}
    \end{minipage}
    \begin{minipage}{0.24\linewidth}
        \centering
        \includegraphics[width=\textwidth]{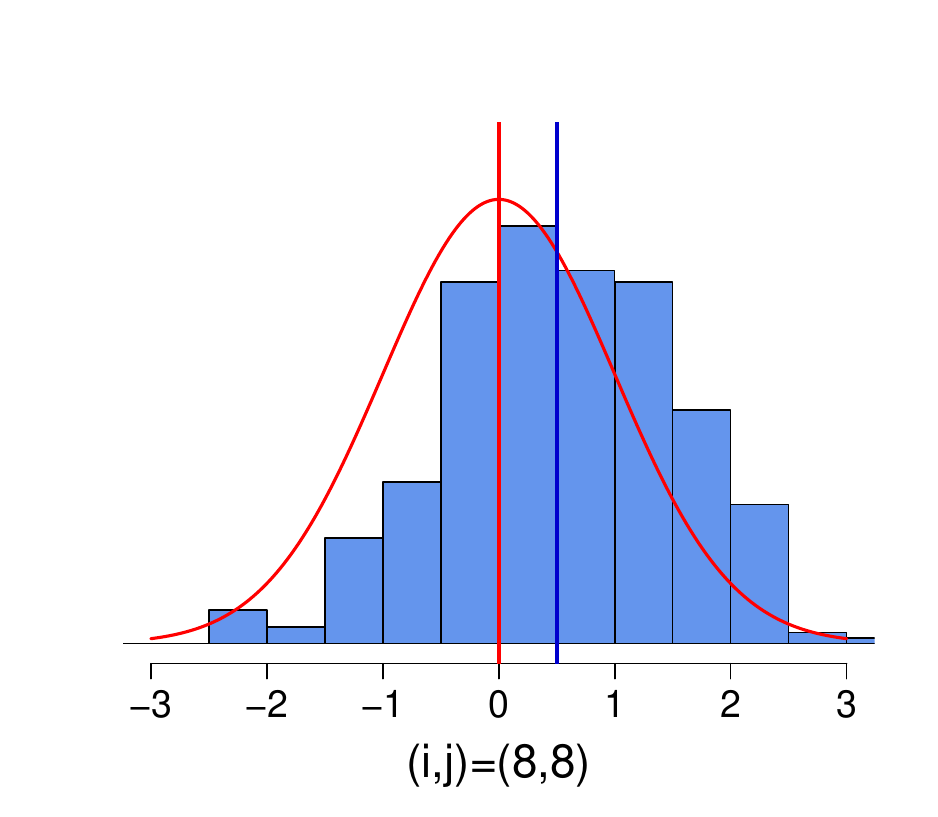}
    \end{minipage}
  \end{minipage}  
    \hspace{1cm}
 \begin{minipage}{0.3\linewidth}
    \begin{minipage}{0.24\linewidth}
        \centering
        \includegraphics[width=\textwidth]{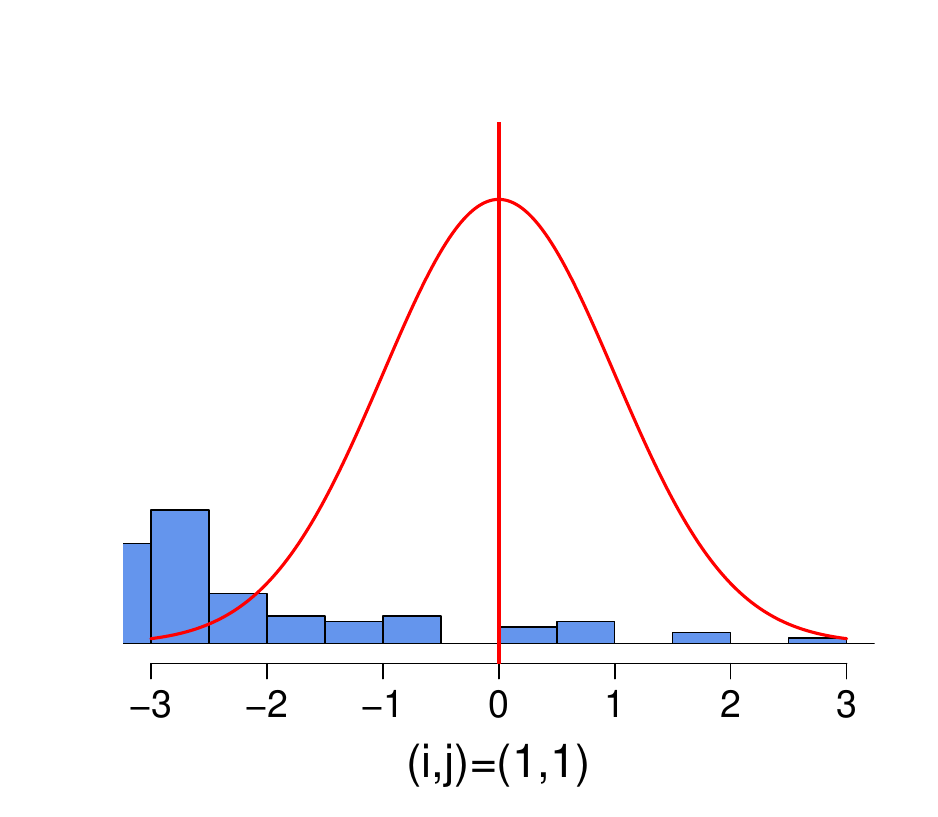}
    \end{minipage}
    \begin{minipage}{0.24\linewidth}
        \centering
        \includegraphics[width=\textwidth]{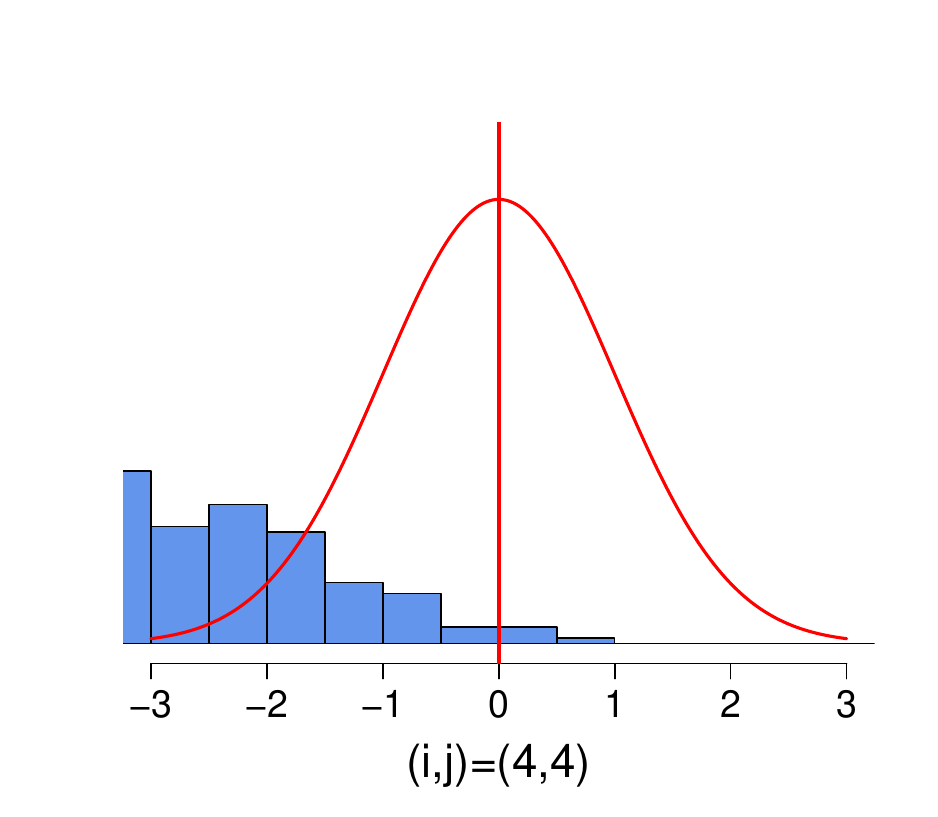}
    \end{minipage}
    \begin{minipage}{0.24\linewidth}
        \centering
        \includegraphics[width=\textwidth]{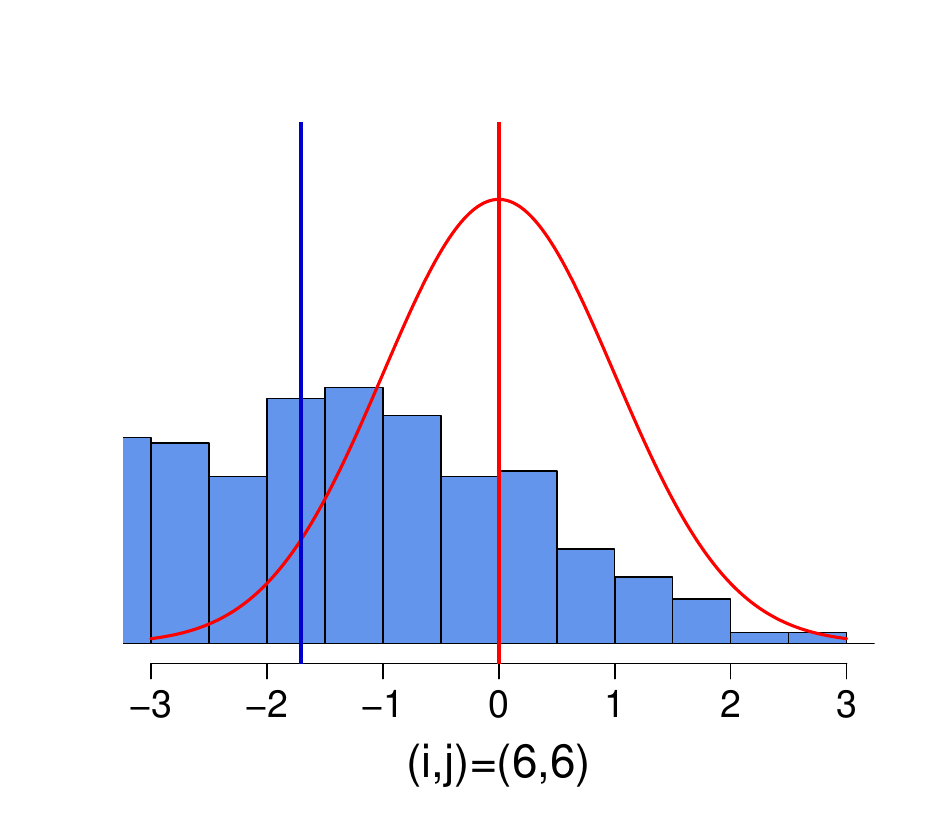}
    \end{minipage}
    \begin{minipage}{0.24\linewidth}
        \centering
        \includegraphics[width=\textwidth]{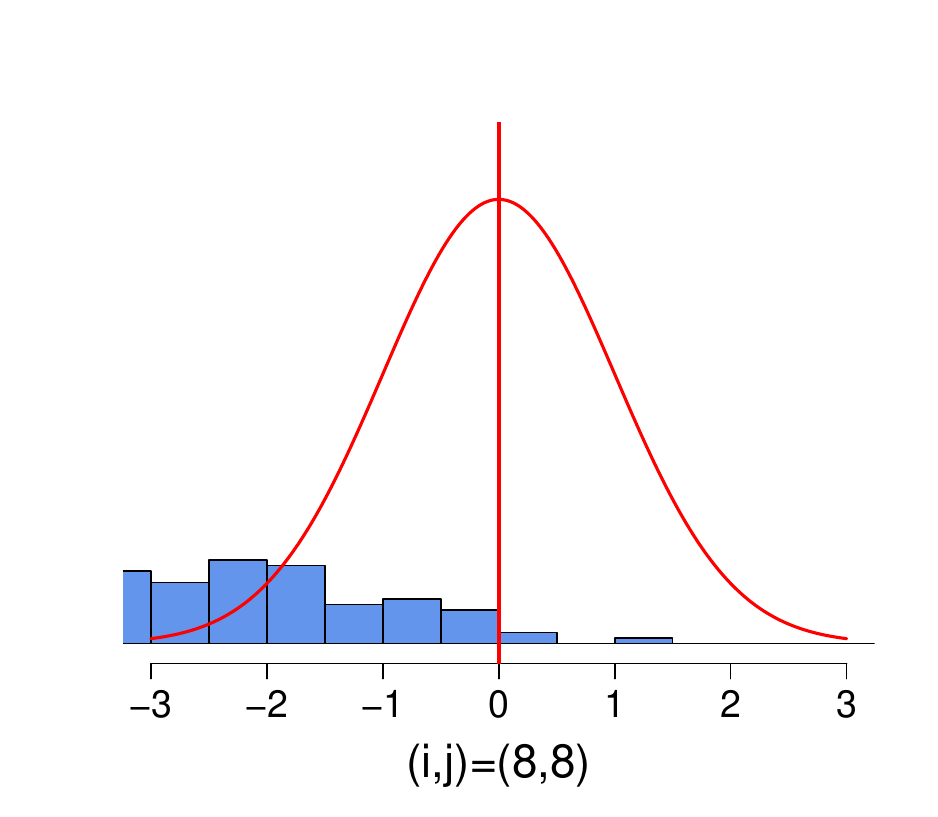}
    \end{minipage}
 \end{minipage}

  \caption*{$n=800, p=200$}
      \vspace{-0.43cm}
 \begin{minipage}{0.3\linewidth}
    \begin{minipage}{0.24\linewidth}
        \centering
        \includegraphics[width=\textwidth]{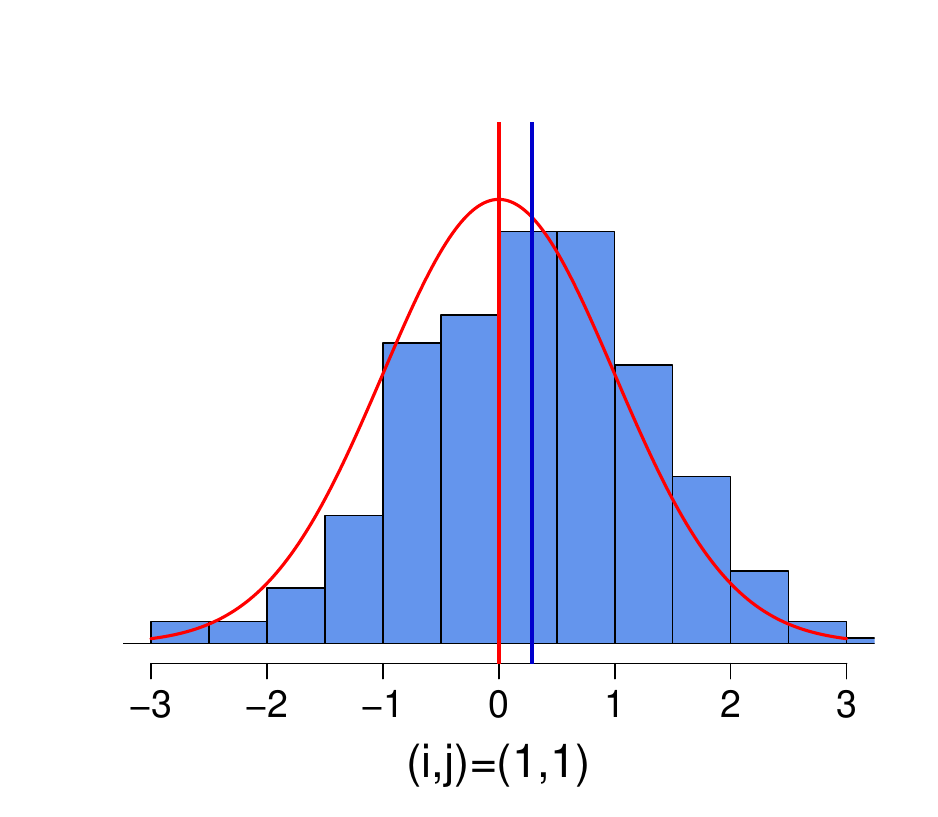}
    \end{minipage}
    \begin{minipage}{0.24\linewidth}
        \centering
        \includegraphics[width=\textwidth]{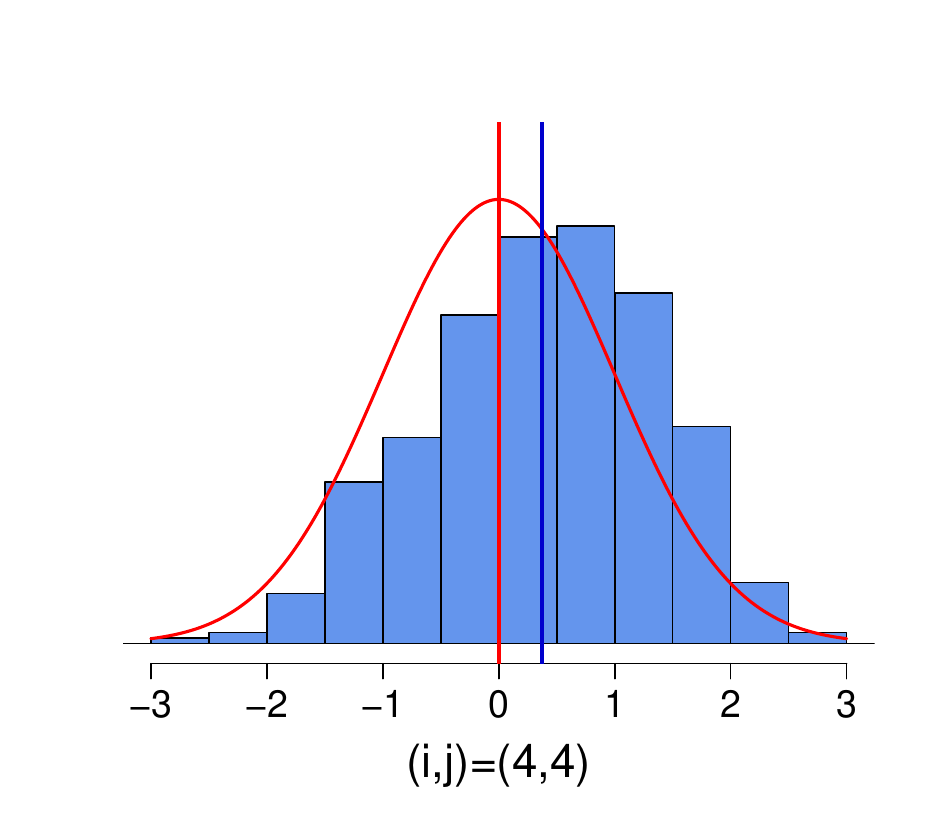}
    \end{minipage}
    \begin{minipage}{0.24\linewidth}
        \centering
        \includegraphics[width=\textwidth]{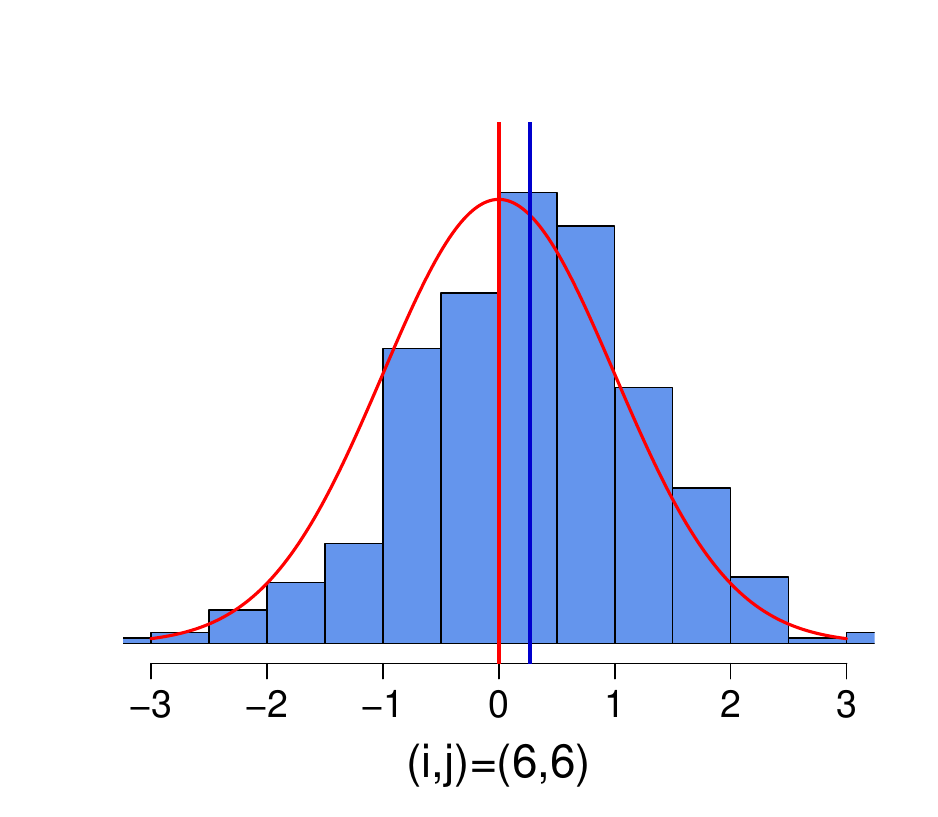}
    \end{minipage}
    \begin{minipage}{0.24\linewidth}
        \centering
        \includegraphics[width=\textwidth]{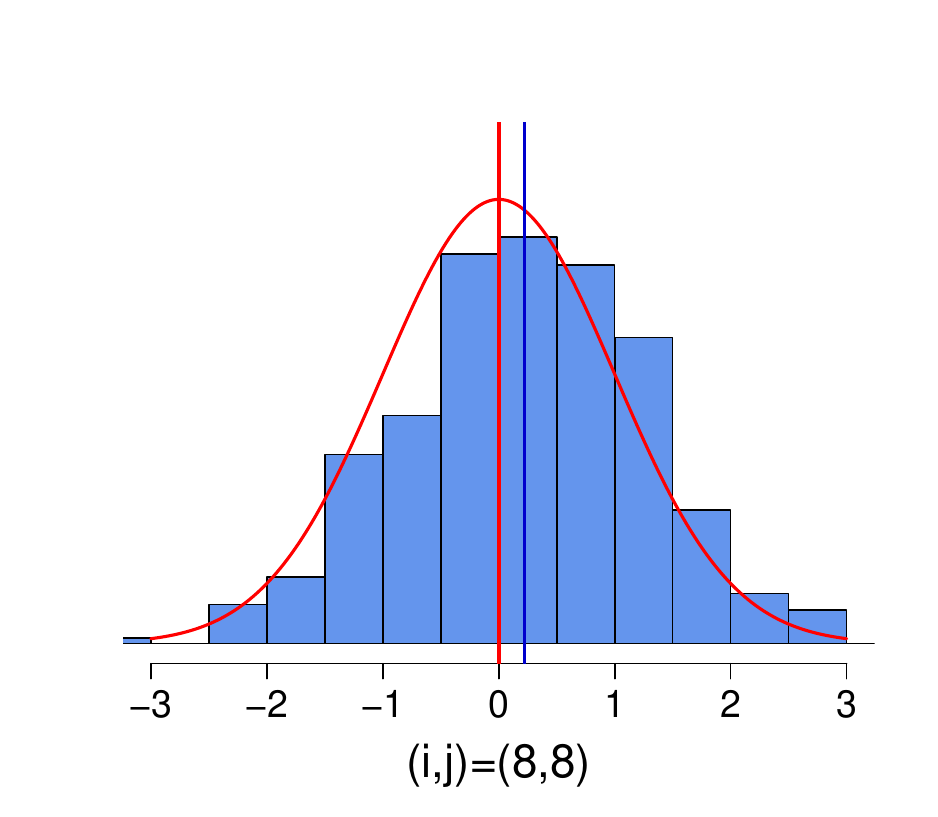}
    \end{minipage}
 \end{minipage} 
     \hspace{1cm}
 \begin{minipage}{0.3\linewidth}
    \begin{minipage}{0.24\linewidth}
        \centering
        \includegraphics[width=\textwidth]{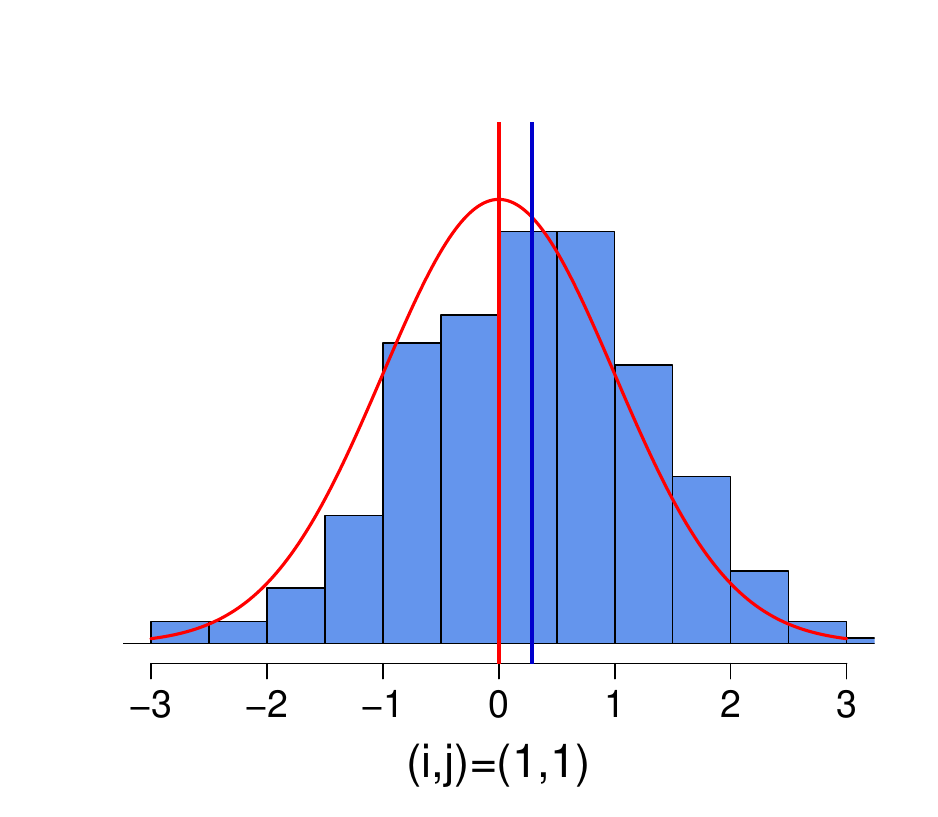}
    \end{minipage}
    \begin{minipage}{0.24\linewidth}
        \centering
        \includegraphics[width=\textwidth]{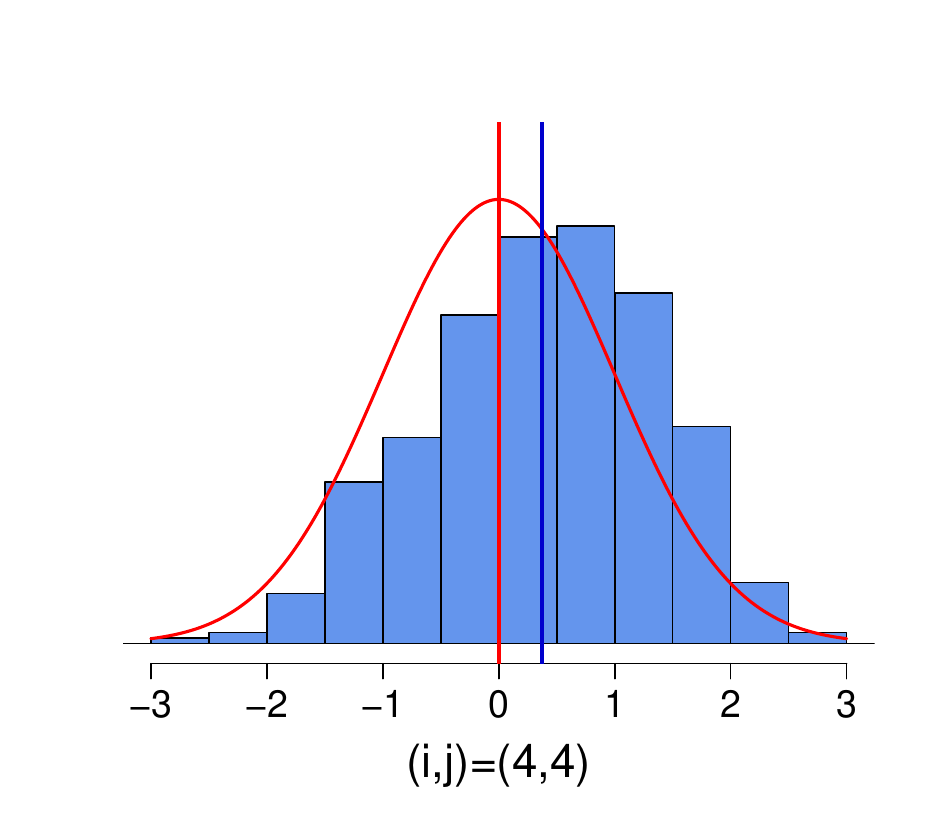}
    \end{minipage}
    \begin{minipage}{0.24\linewidth}
        \centering
        \includegraphics[width=\textwidth]{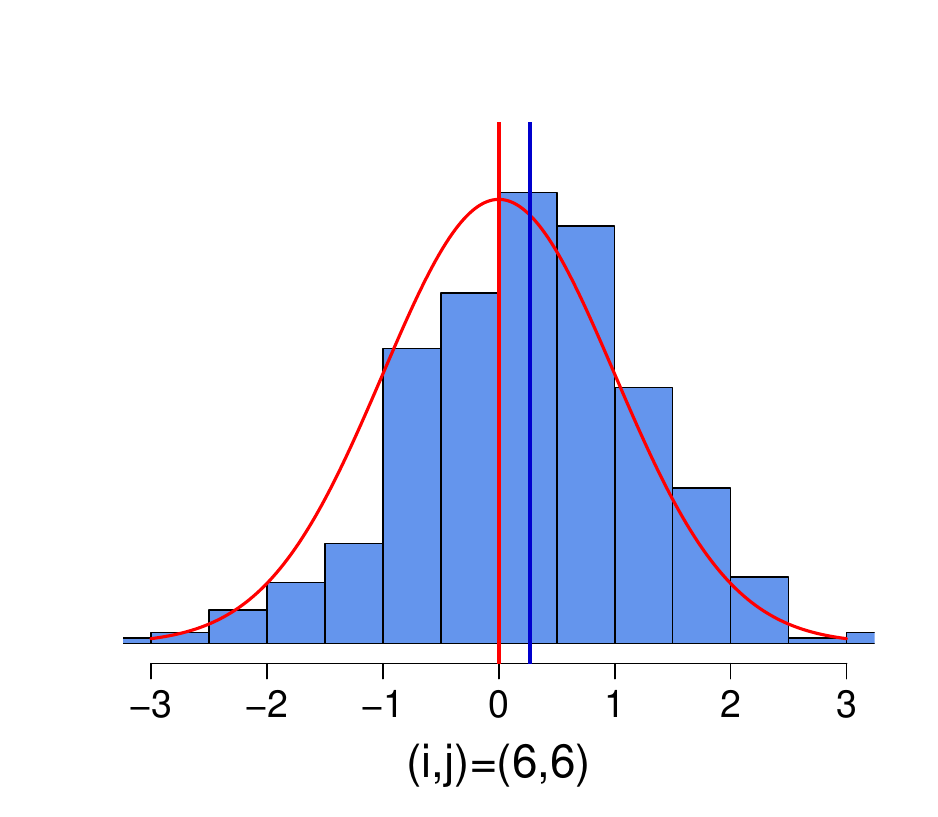}
    \end{minipage}
    \begin{minipage}{0.24\linewidth}
        \centering
        \includegraphics[width=\textwidth]{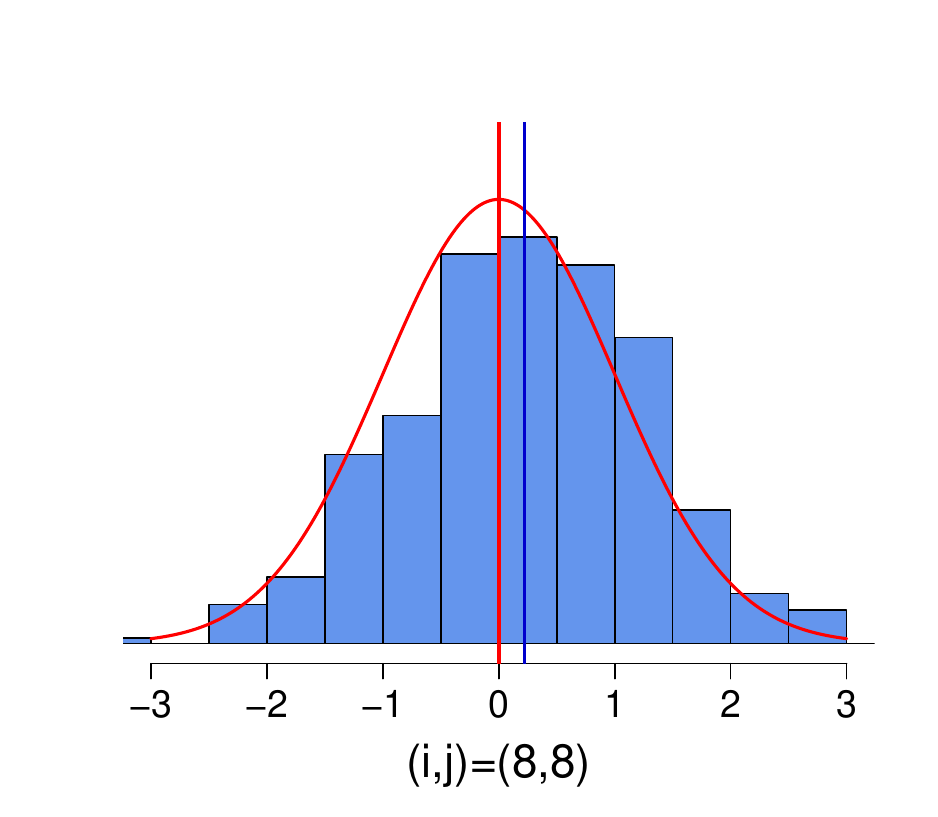}
    \end{minipage}
 \end{minipage}   
      \hspace{1cm}
 \begin{minipage}{0.3\linewidth}
    \begin{minipage}{0.24\linewidth}
        \centering
        \includegraphics[width=\textwidth]{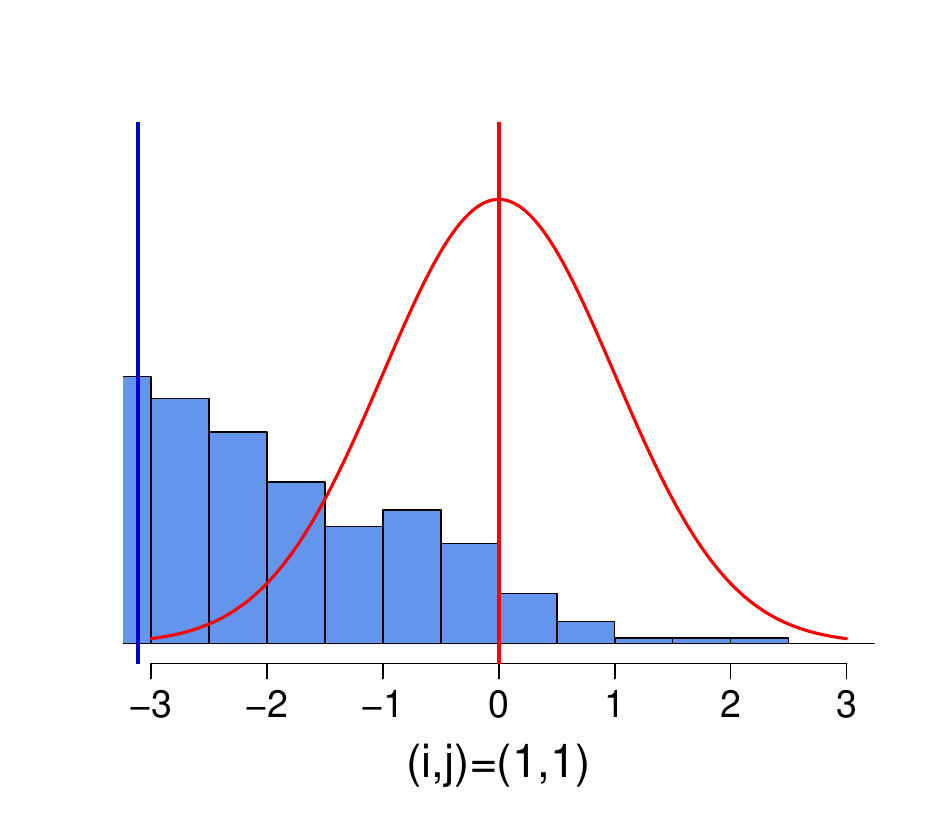}
    \end{minipage}
    \begin{minipage}{0.24\linewidth}
        \centering
        \includegraphics[width=\textwidth]{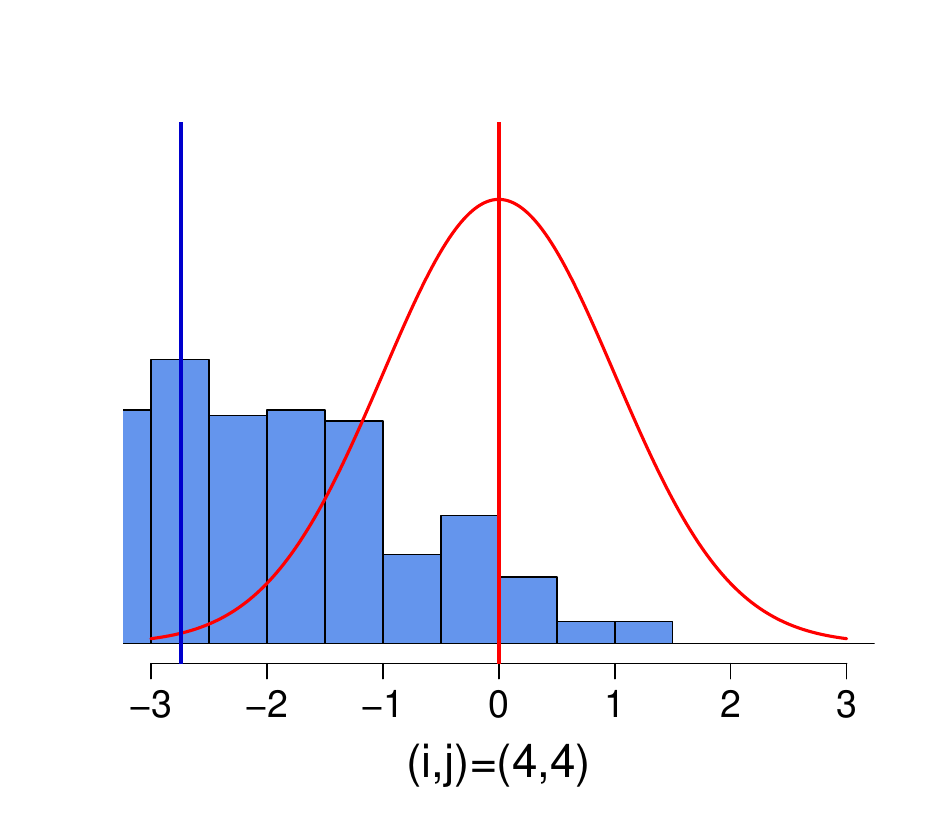}
    \end{minipage}
    \begin{minipage}{0.24\linewidth}
        \centering
        \includegraphics[width=\textwidth]{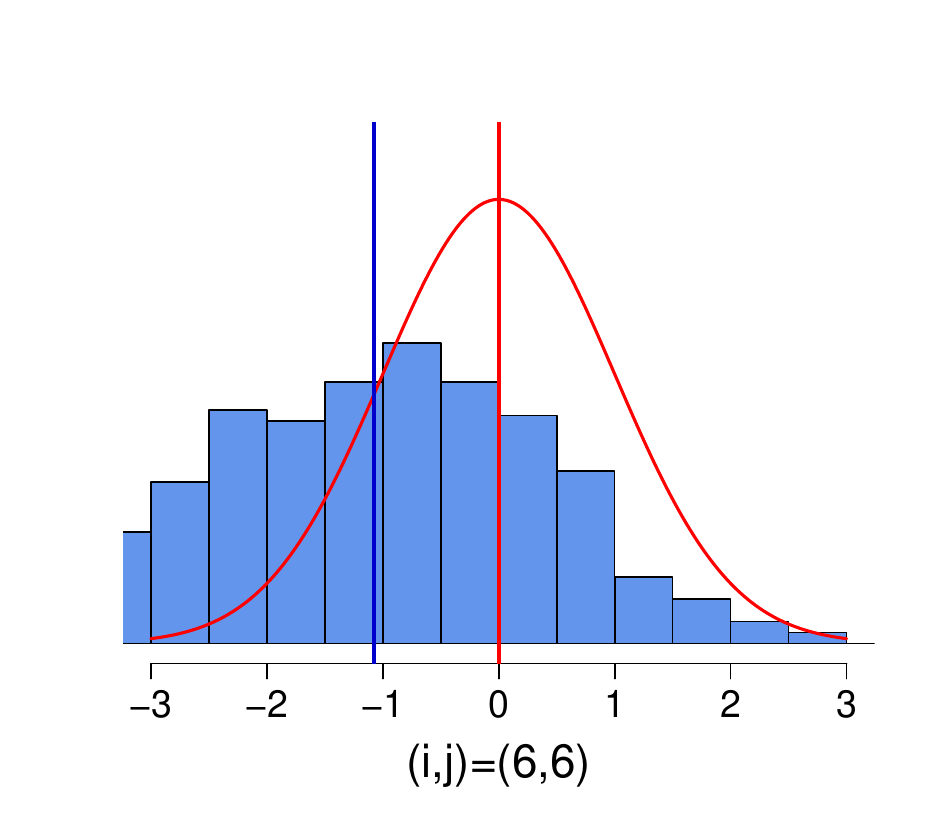}
    \end{minipage}
    \begin{minipage}{0.24\linewidth}
        \centering
        \includegraphics[width=\textwidth]{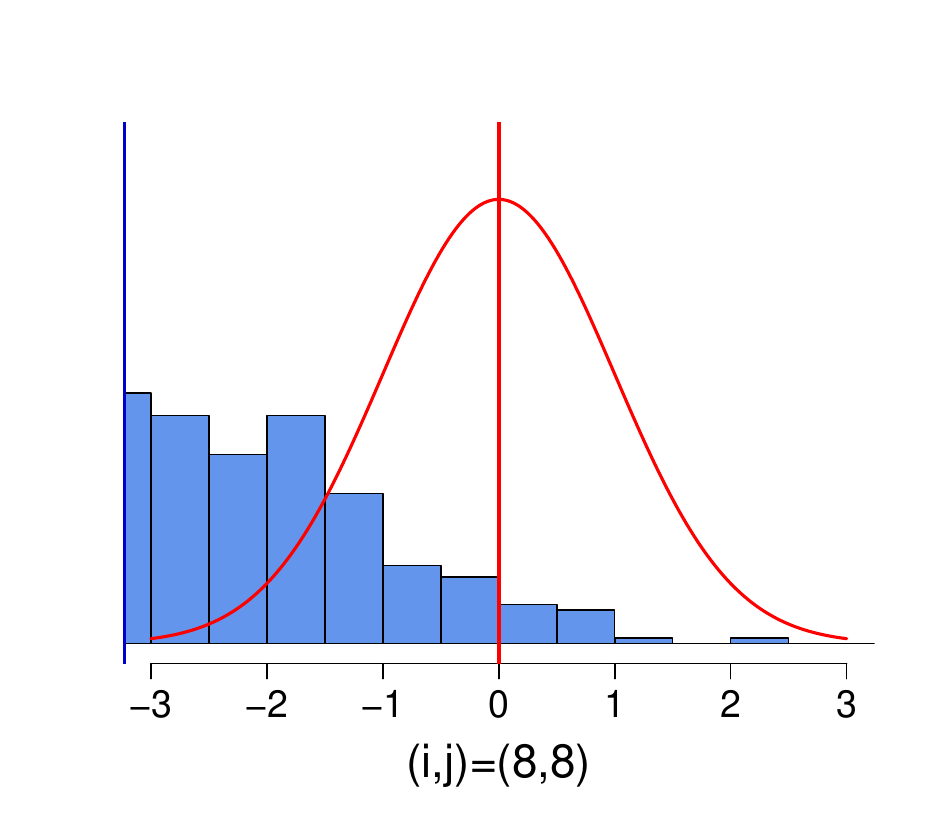}
    \end{minipage}
     \end{minipage}   
     
 \caption*{$n=200, p=400$}
     \vspace{-0.43cm}
 \begin{minipage}{0.3\linewidth}
    \begin{minipage}{0.24\linewidth}
        \centering
        \includegraphics[width=\textwidth]{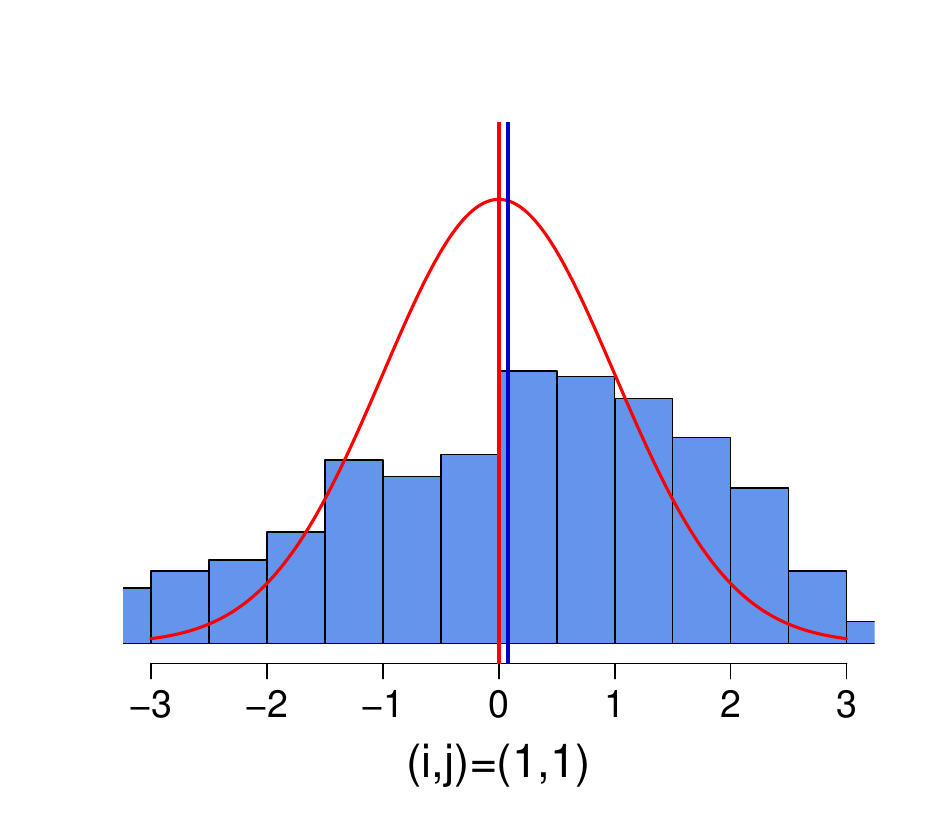}
    \end{minipage}
    \begin{minipage}{0.24\linewidth}
        \centering
        \includegraphics[width=\textwidth]{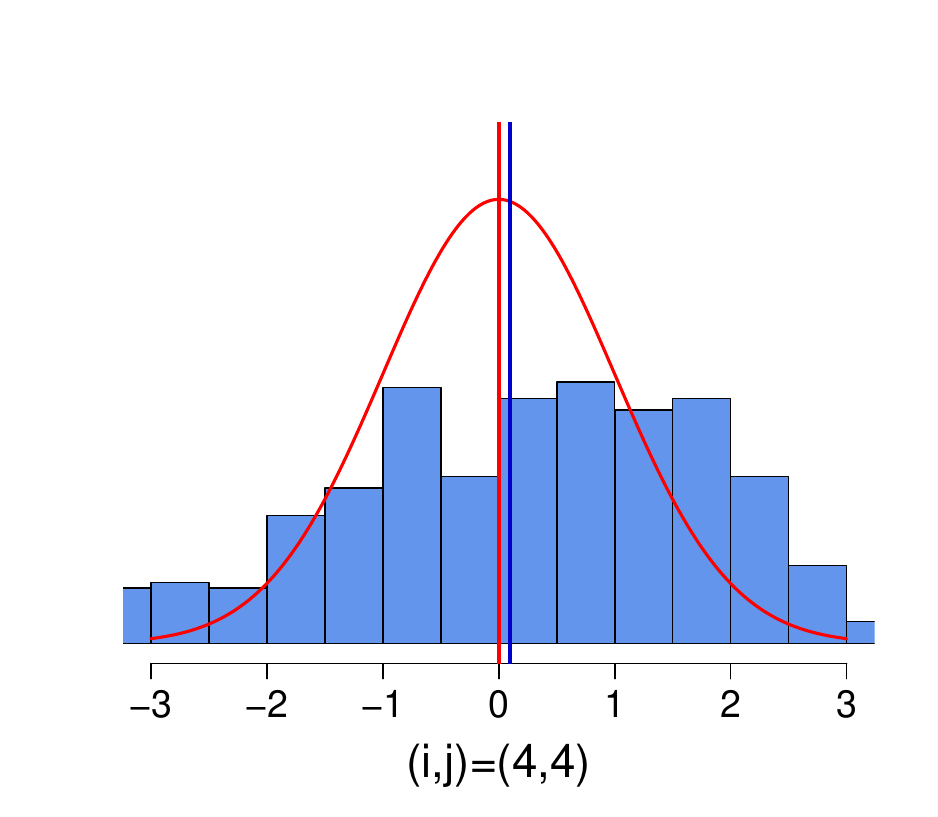}
    \end{minipage}
    \begin{minipage}{0.24\linewidth}
        \centering
        \includegraphics[width=\textwidth]{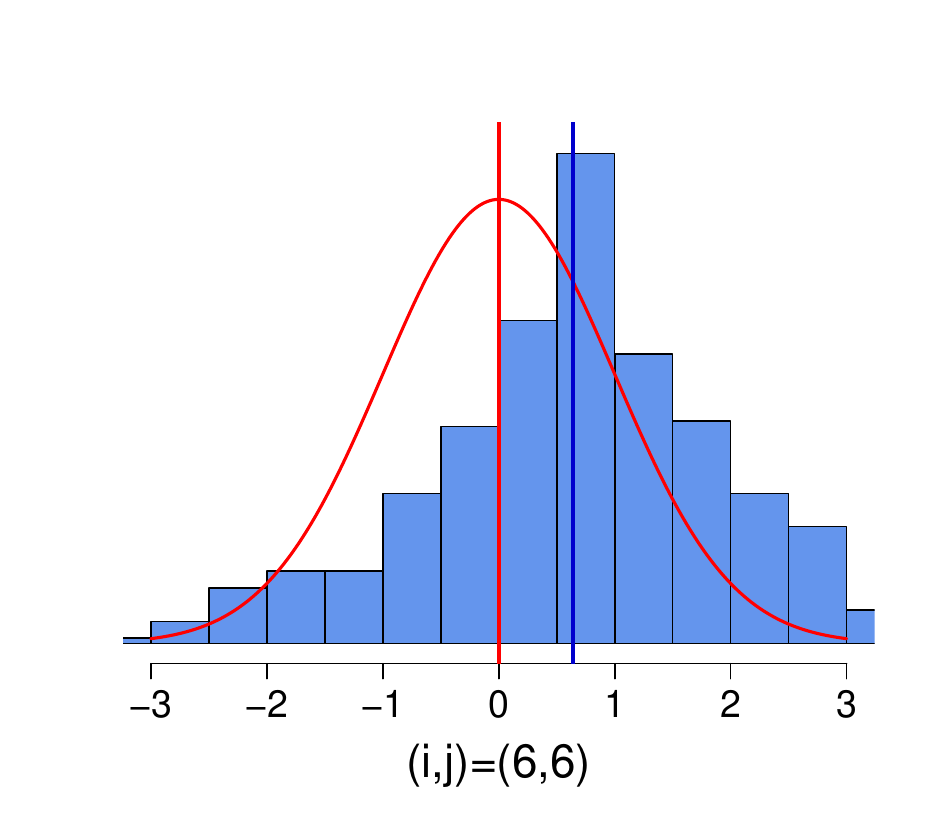}
    \end{minipage}
    \begin{minipage}{0.24\linewidth}
        \centering
        \includegraphics[width=\textwidth]{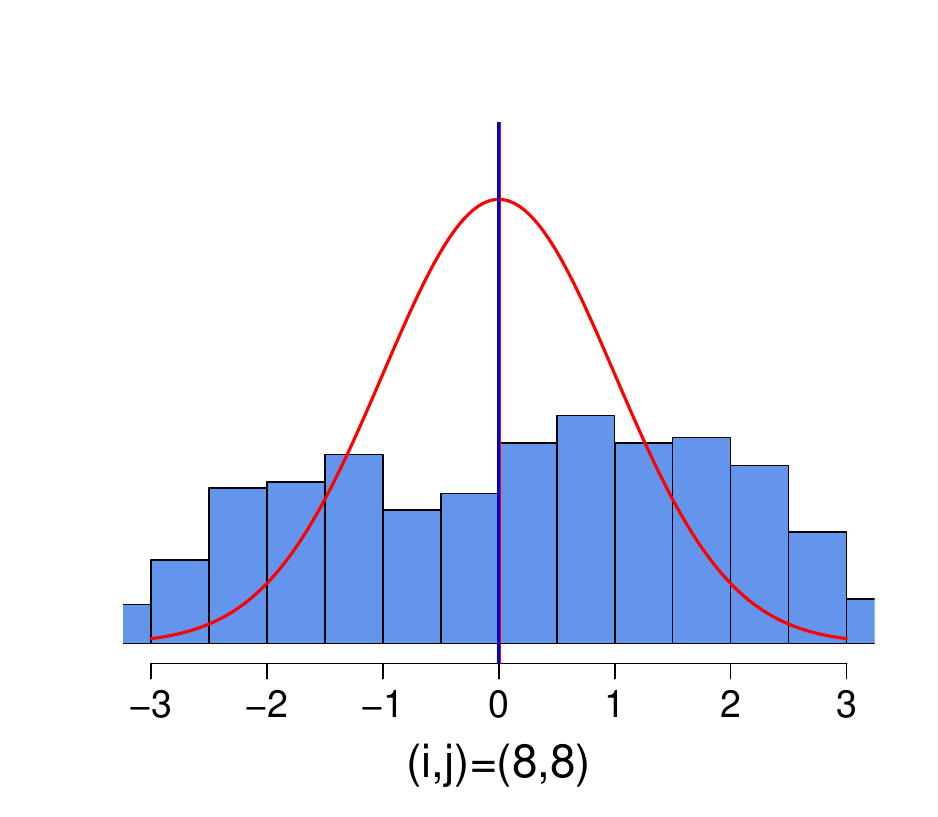}
    \end{minipage}
 \end{minipage}
 \hspace{1cm}
 \begin{minipage}{0.3\linewidth}
    \begin{minipage}{0.24\linewidth}
        \centering
        \includegraphics[width=\textwidth]{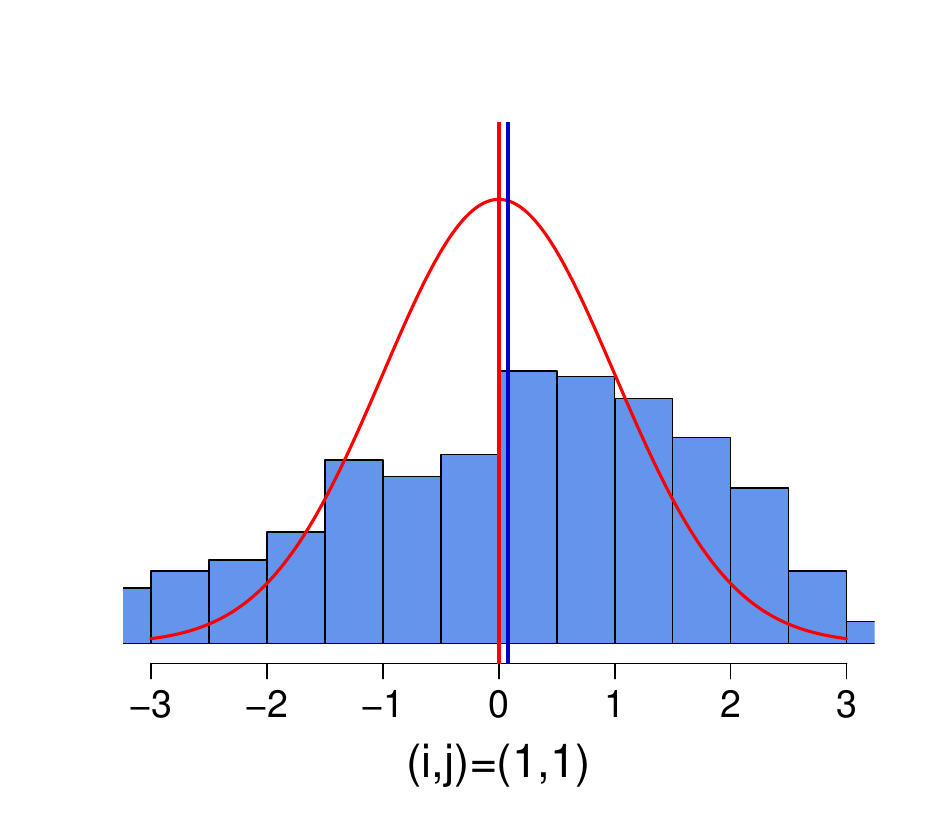}
    \end{minipage}
    \begin{minipage}{0.24\linewidth}
        \centering
        \includegraphics[width=\textwidth]{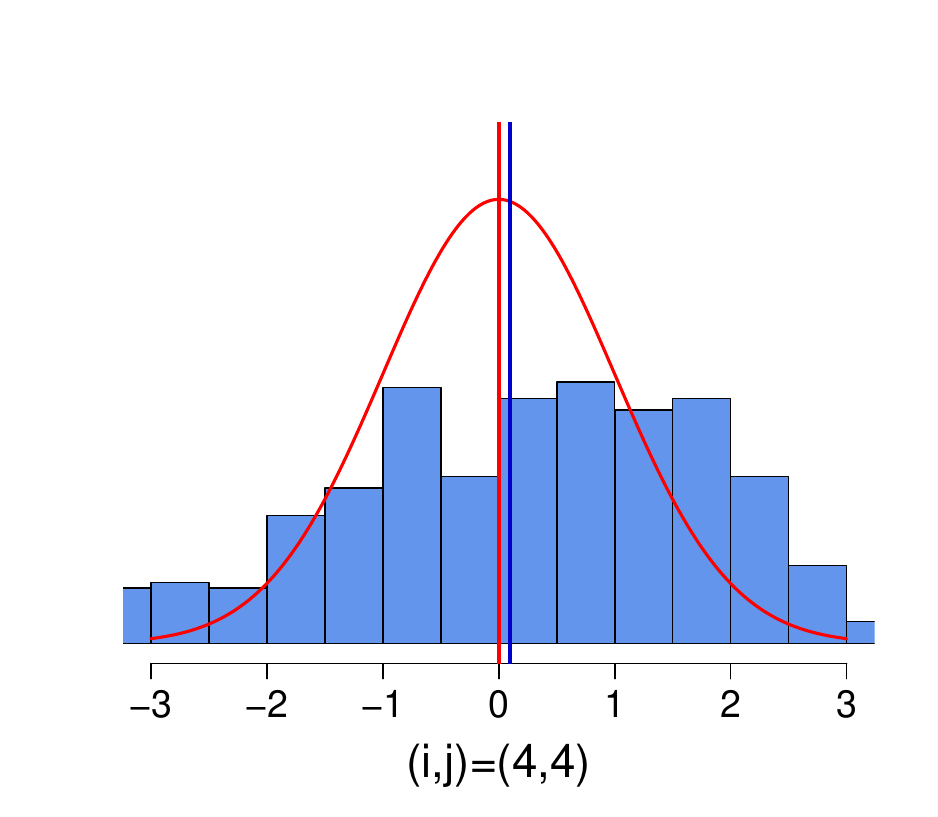}
    \end{minipage}
    \begin{minipage}{0.24\linewidth}
        \centering
        \includegraphics[width=\textwidth]{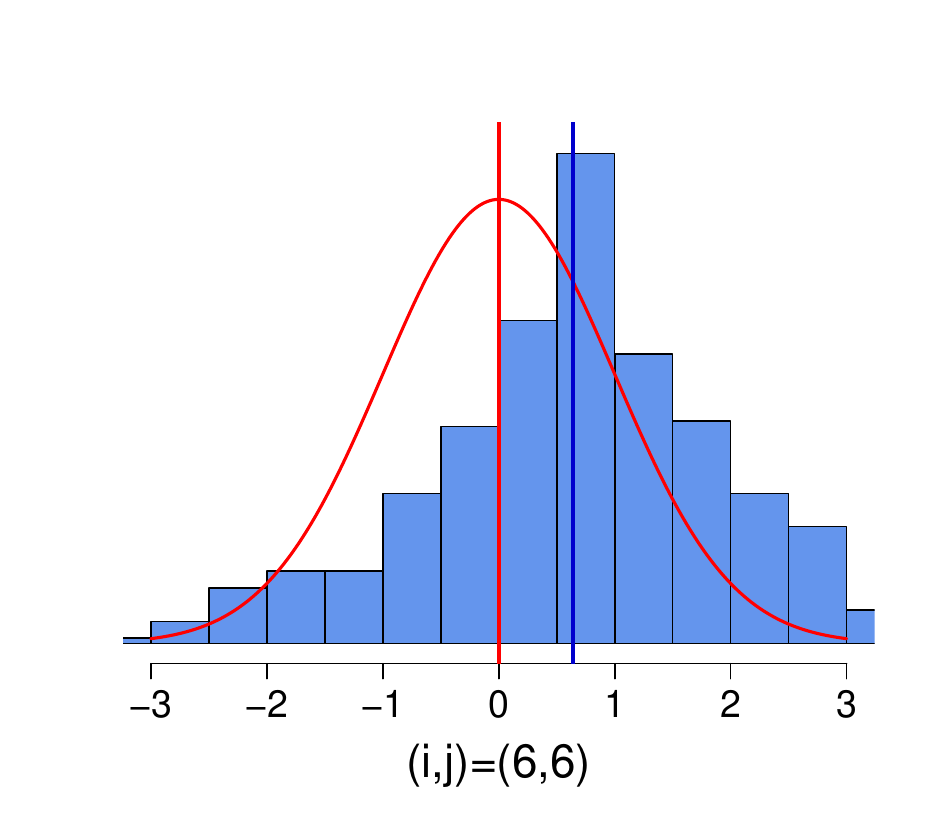}
    \end{minipage}
    \begin{minipage}{0.24\linewidth}
        \centering
        \includegraphics[width=\textwidth]{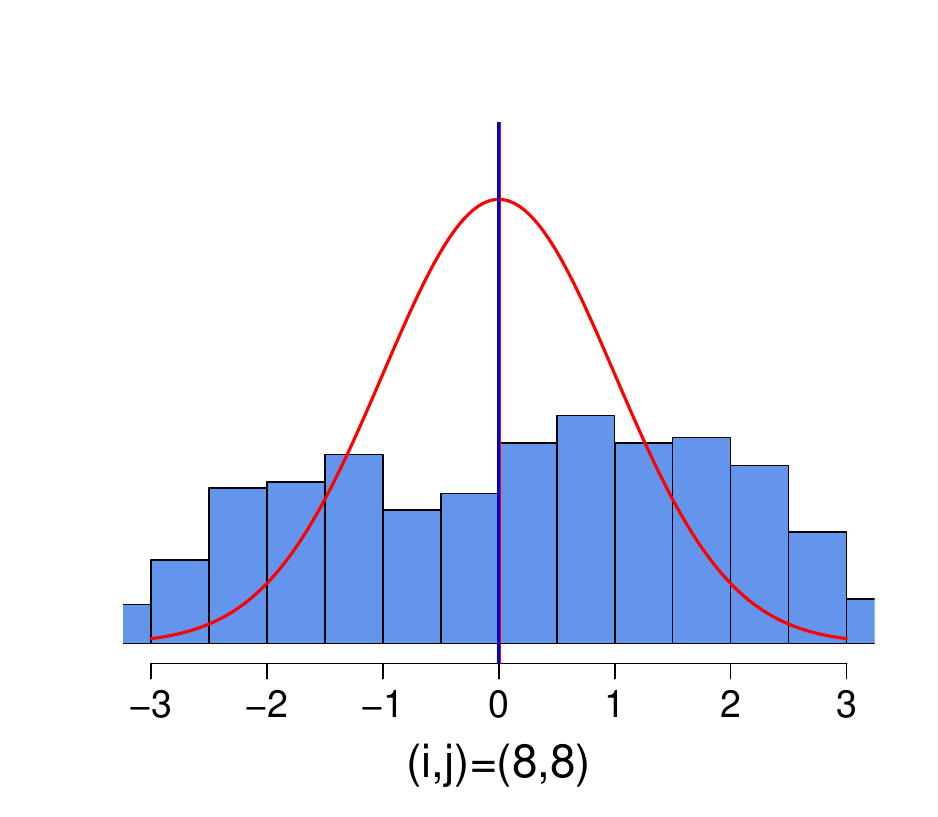}
    \end{minipage}    
 \end{minipage}
  \hspace{1cm}
 \begin{minipage}{0.3\linewidth}
     \begin{minipage}{0.24\linewidth}
        \centering
        \includegraphics[width=\textwidth]{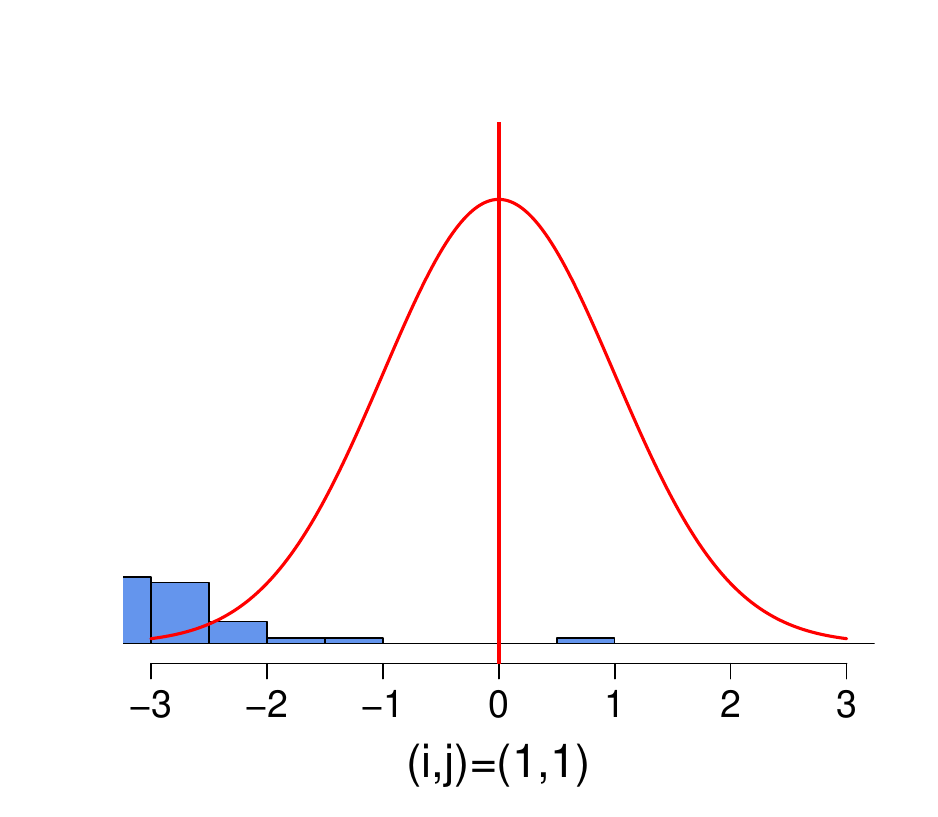}
    \end{minipage}
    \begin{minipage}{0.24\linewidth}
        \centering
        \includegraphics[width=\textwidth]{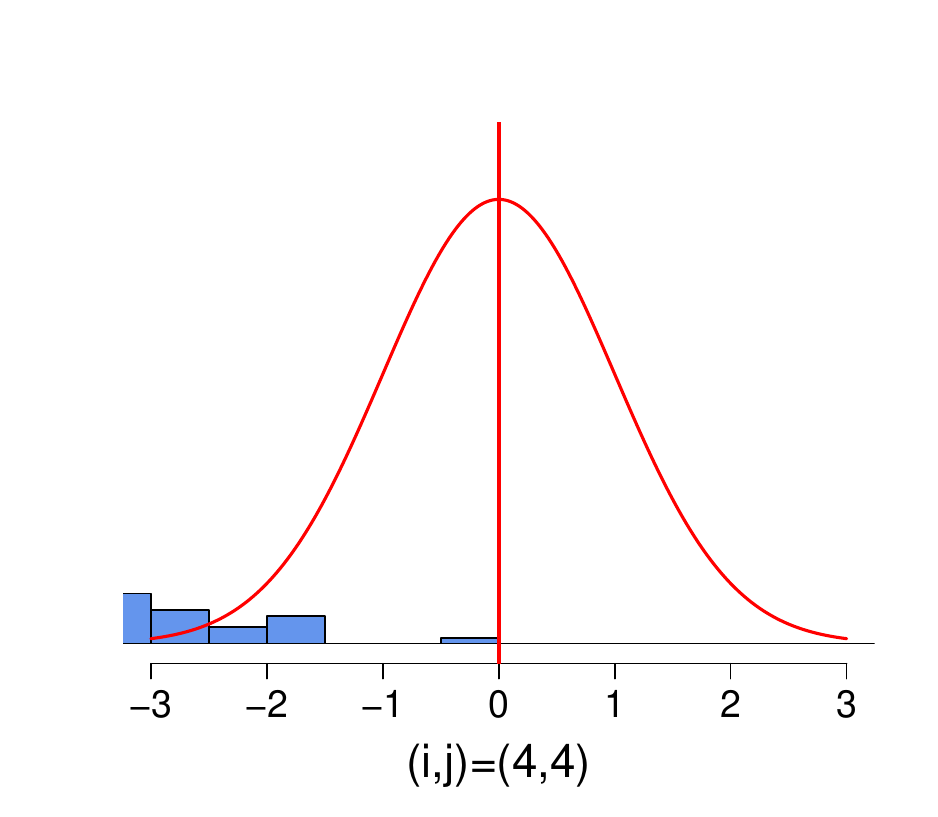}
    \end{minipage}
    \begin{minipage}{0.24\linewidth}
        \centering
        \includegraphics[width=\textwidth]{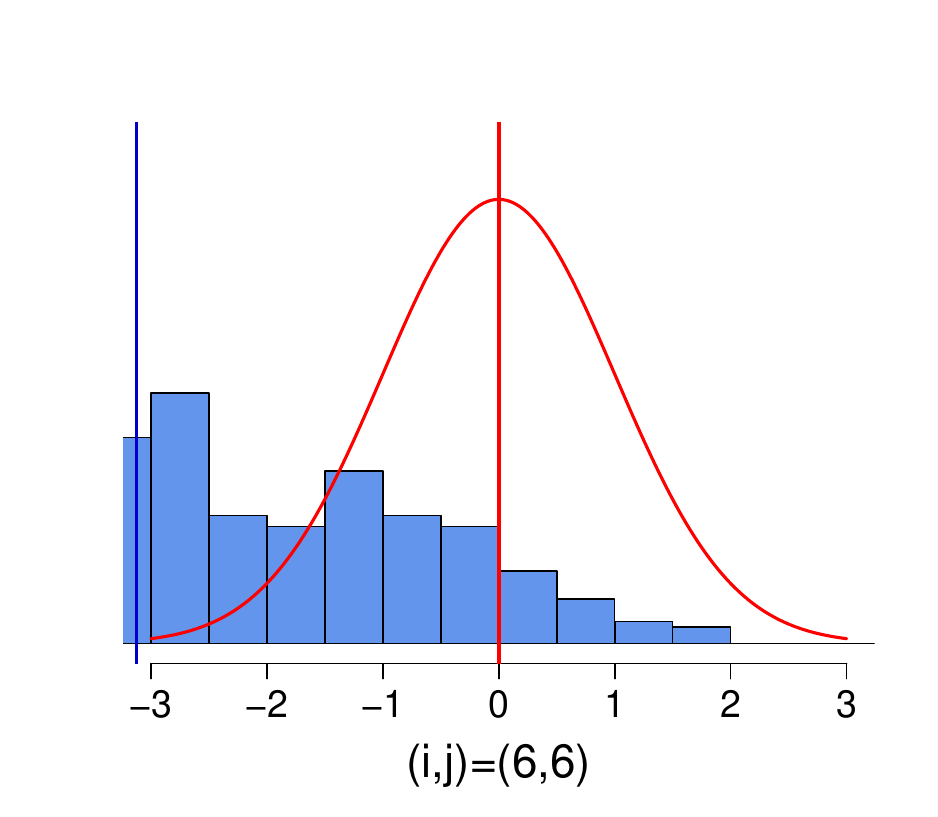}
    \end{minipage}
    \begin{minipage}{0.24\linewidth}
        \centering
        \includegraphics[width=\textwidth]{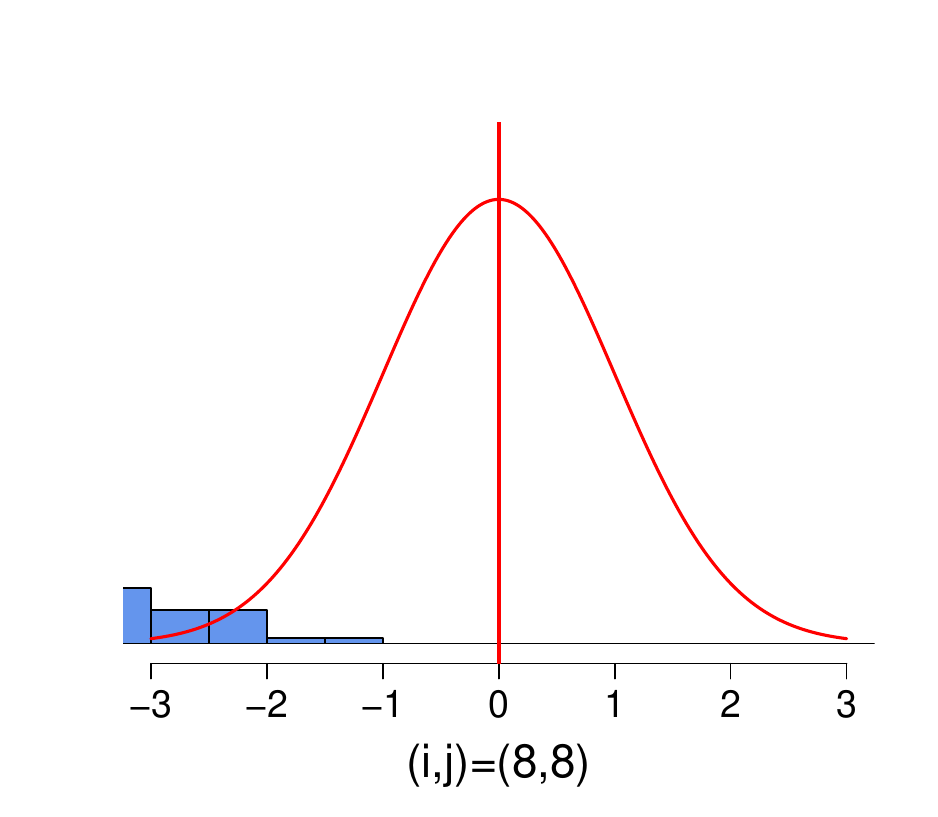}
    \end{minipage}
 \end{minipage}

  \caption*{$n=400, p=400$}
      \vspace{-0.43cm}
 \begin{minipage}{0.3\linewidth}
    \begin{minipage}{0.24\linewidth}
        \centering
        \includegraphics[width=\textwidth]{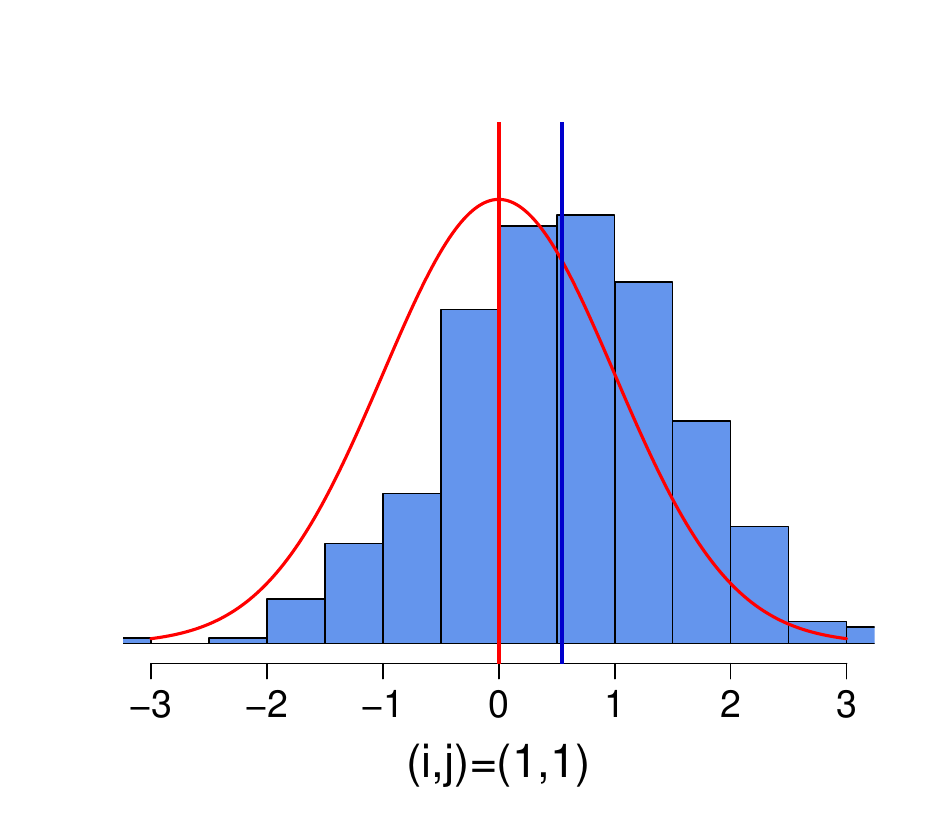}
    \end{minipage}
    \begin{minipage}{0.24\linewidth}
        \centering
        \includegraphics[width=\textwidth]{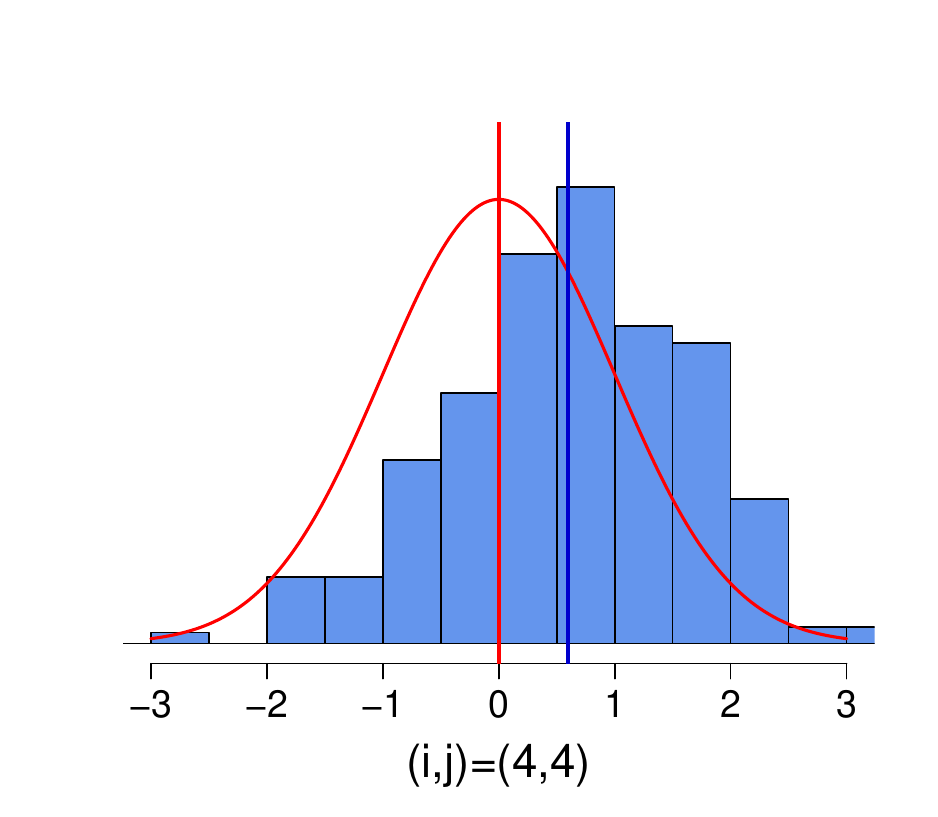}
    \end{minipage}
    \begin{minipage}{0.24\linewidth}
        \centering
        \includegraphics[width=\textwidth]{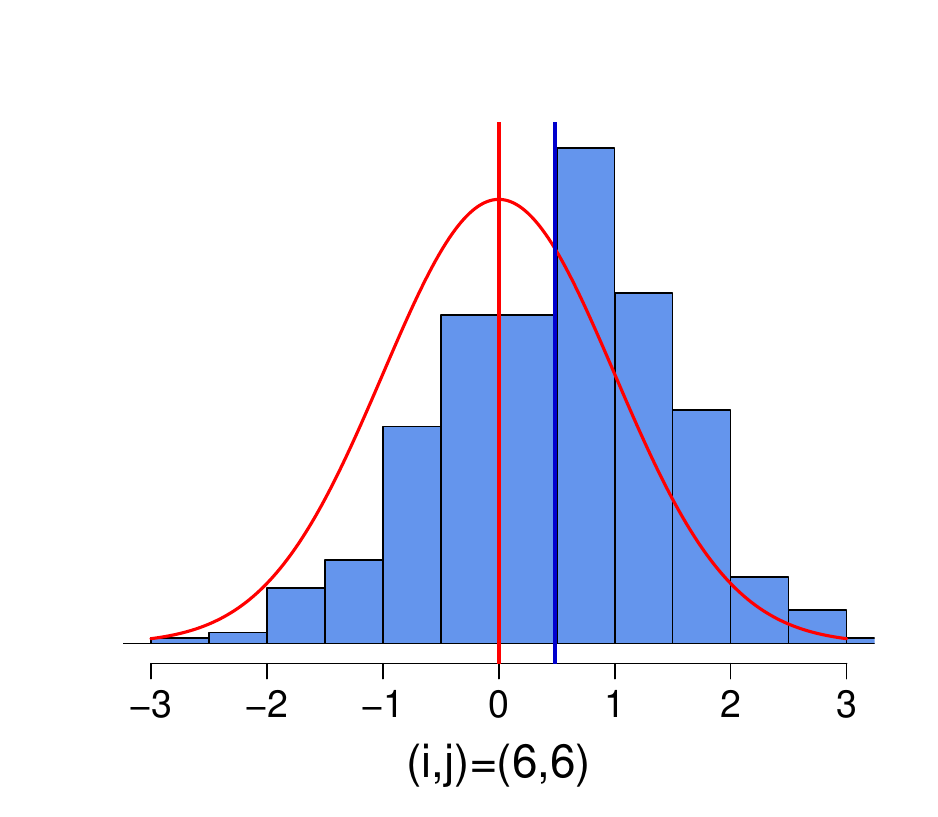}
    \end{minipage}
    \begin{minipage}{0.24\linewidth}
        \centering
        \includegraphics[width=\textwidth]{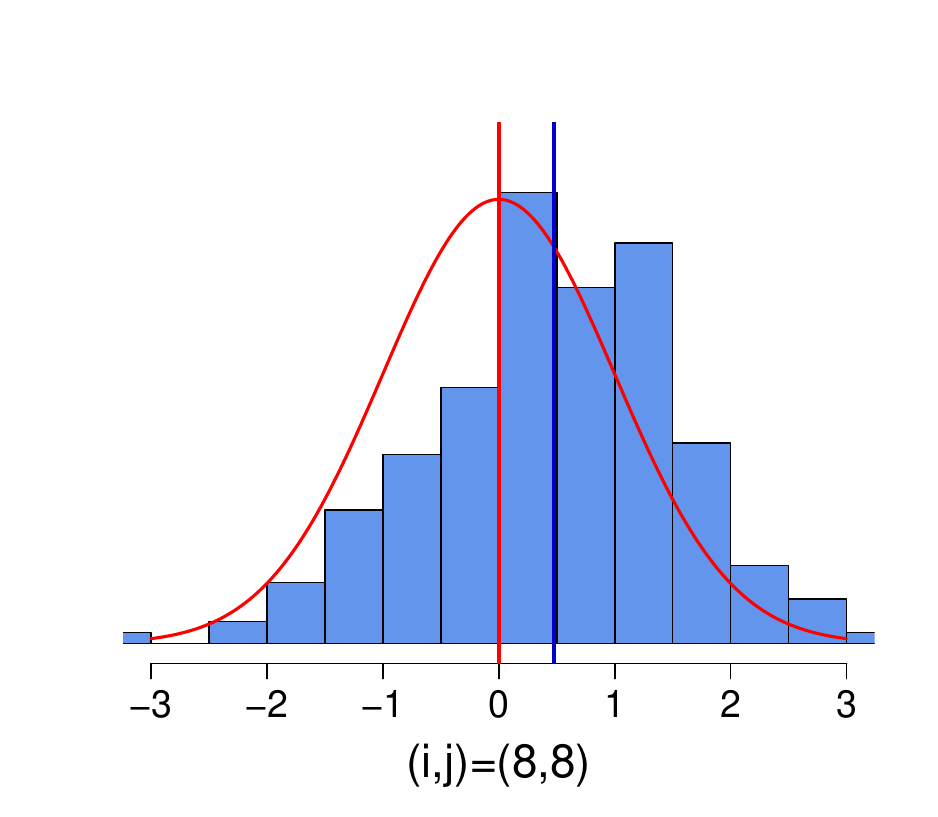}
    \end{minipage}
 \end{minipage}  
     \hspace{1cm}
 \begin{minipage}{0.3\linewidth}
    \begin{minipage}{0.24\linewidth}
        \centering
        \includegraphics[width=\textwidth]{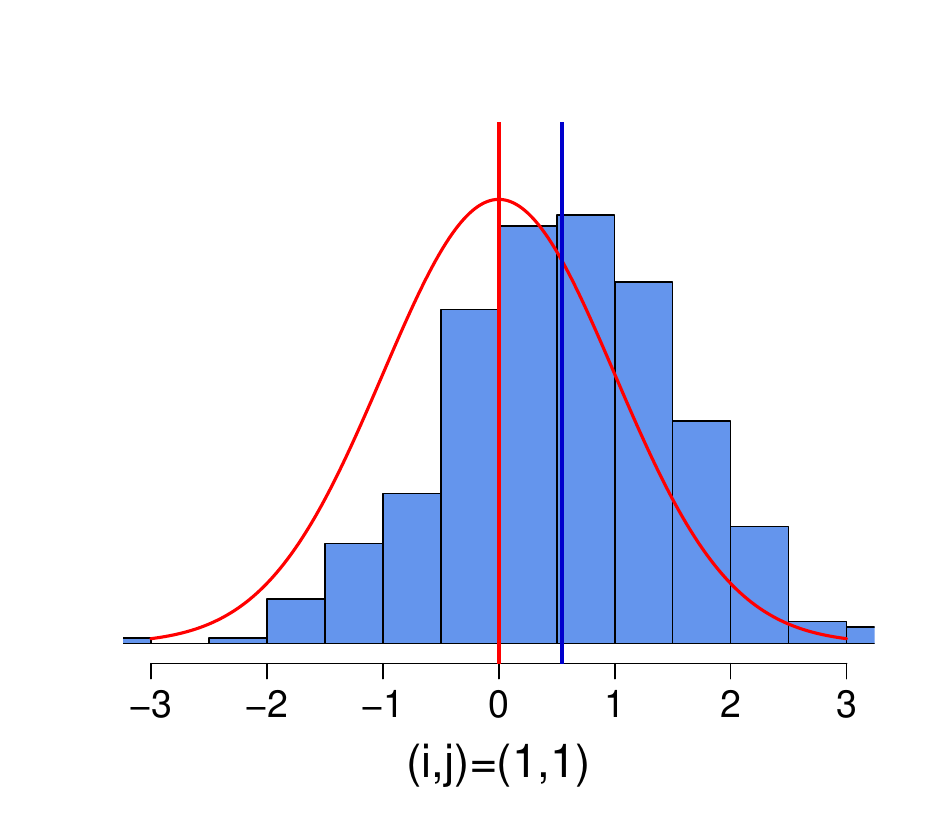}
    \end{minipage}
    \begin{minipage}{0.24\linewidth}
        \centering
        \includegraphics[width=\textwidth]{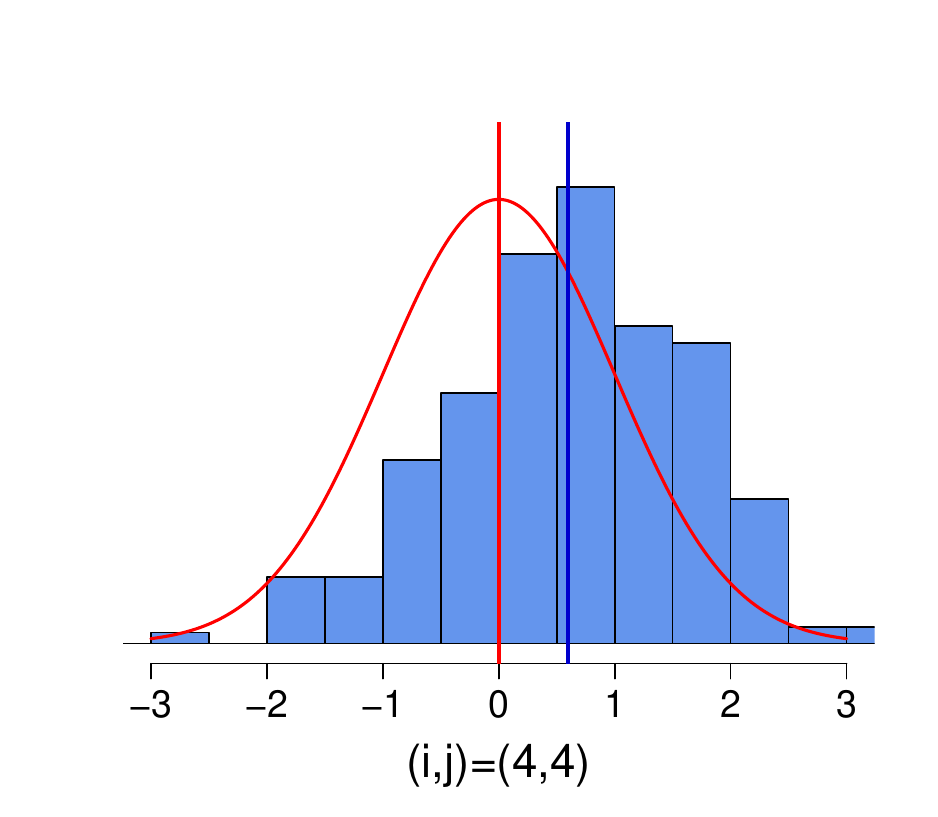}
    \end{minipage}
    \begin{minipage}{0.24\linewidth}
        \centering
        \includegraphics[width=\textwidth]{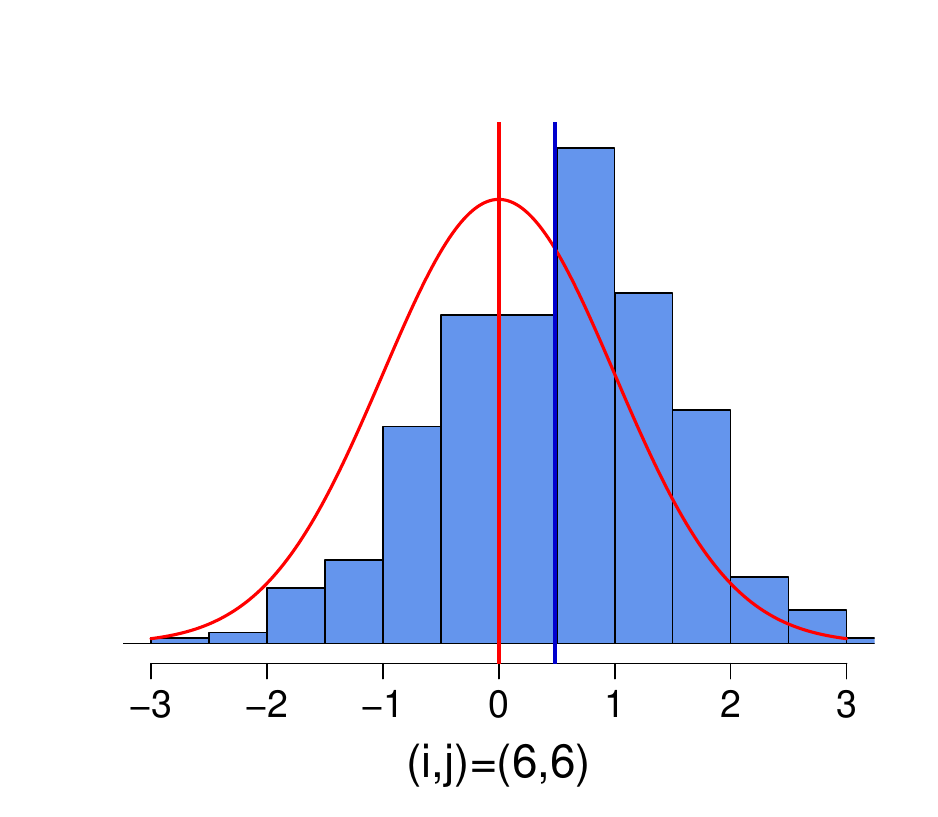}
    \end{minipage}
    \begin{minipage}{0.24\linewidth}
        \centering
        \includegraphics[width=\textwidth]{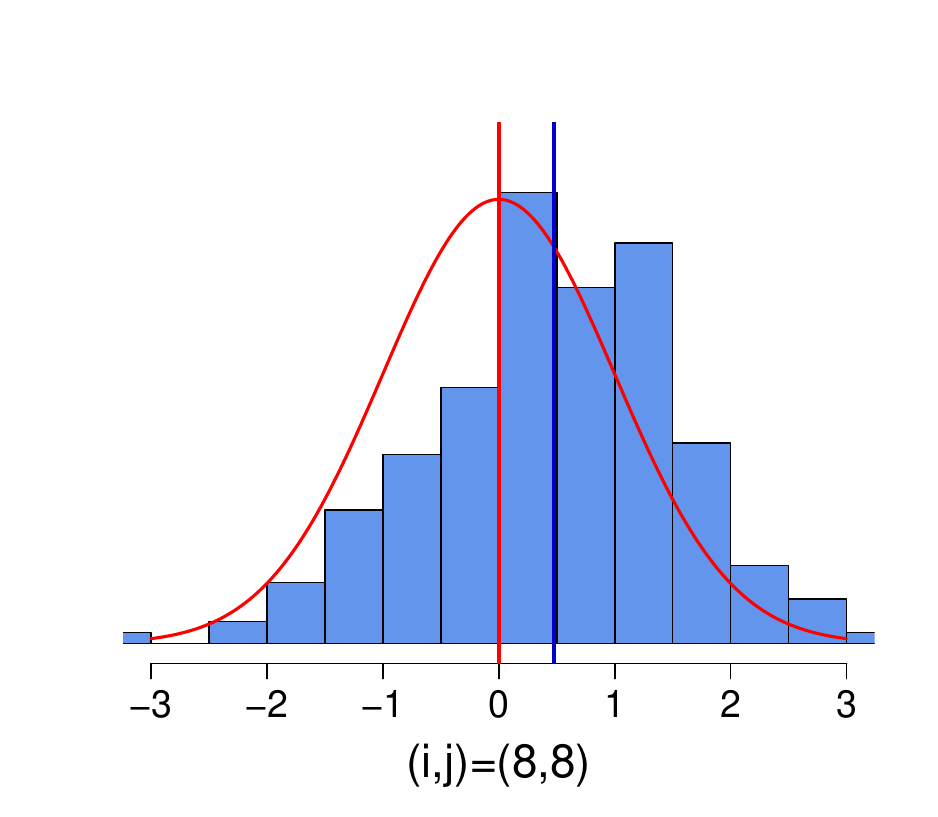}
    \end{minipage}
  \end{minipage}  
    \hspace{1cm}
 \begin{minipage}{0.3\linewidth}
    \begin{minipage}{0.24\linewidth}
        \centering
        \includegraphics[width=\textwidth]{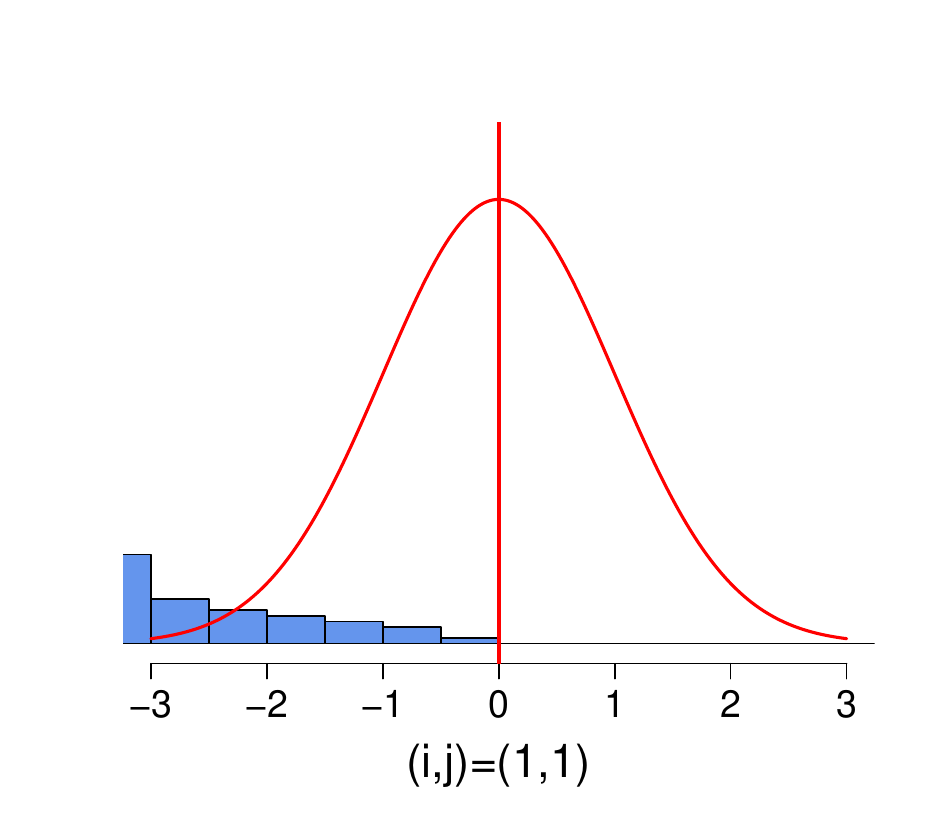}
    \end{minipage}
    \begin{minipage}{0.24\linewidth}
        \centering
        \includegraphics[width=\textwidth]{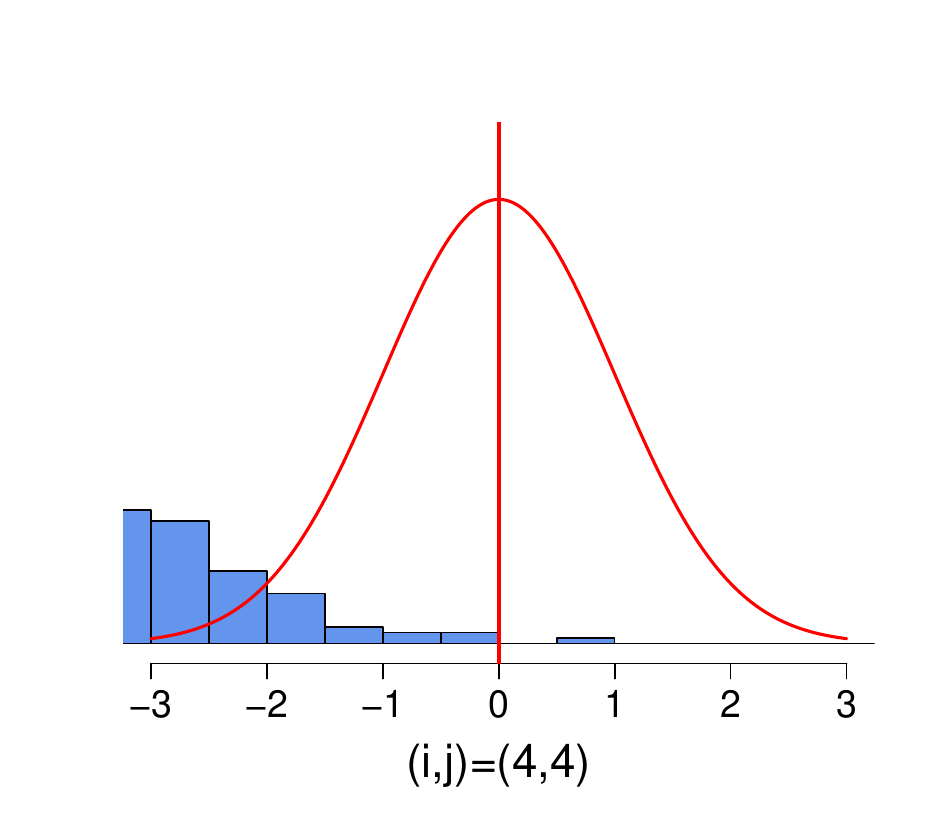}
    \end{minipage}
    \begin{minipage}{0.24\linewidth}
        \centering
        \includegraphics[width=\textwidth]{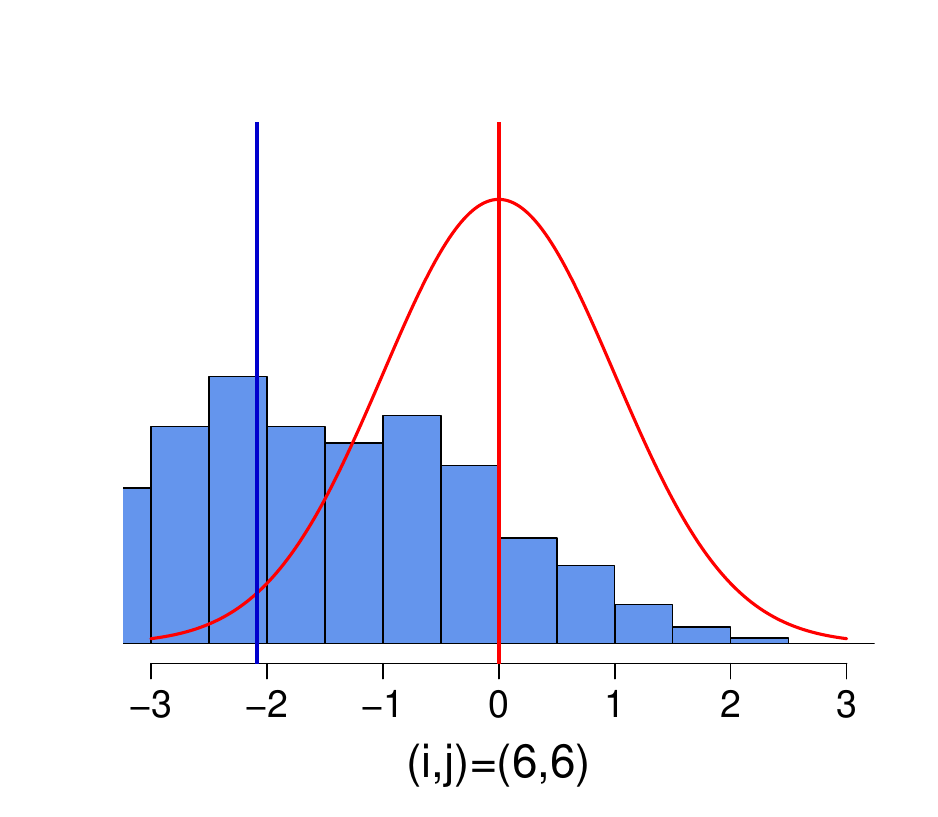}
    \end{minipage}
    \begin{minipage}{0.24\linewidth}
        \centering
        \includegraphics[width=\textwidth]{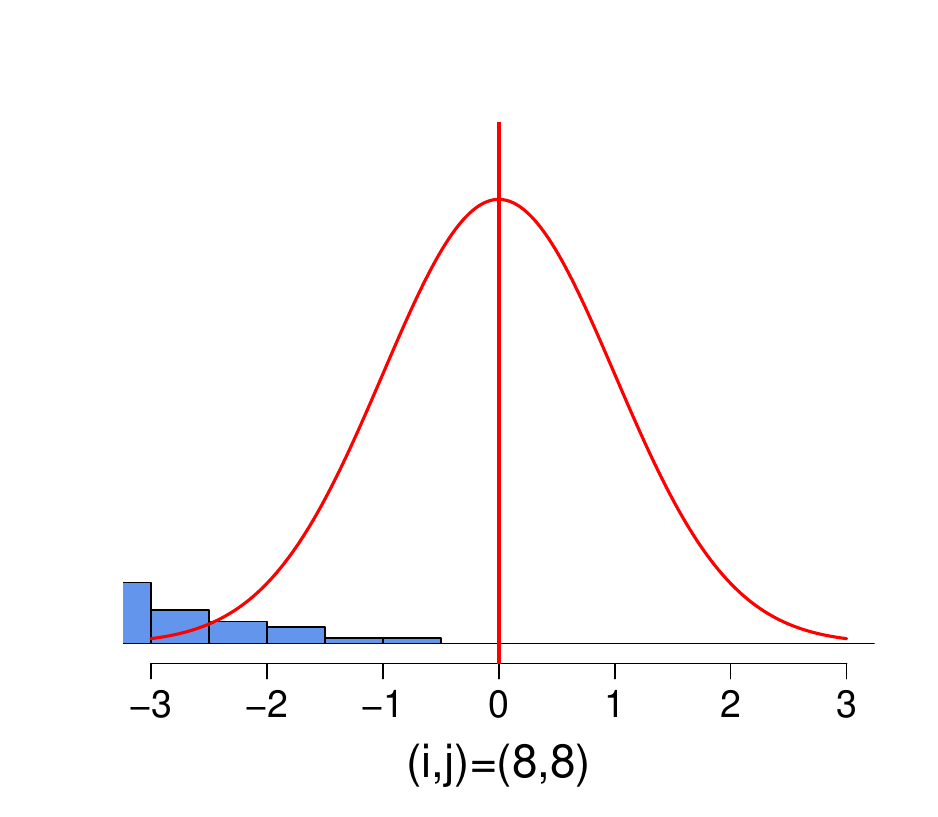}
    \end{minipage}
 \end{minipage}

  \caption*{$n=800, p=400$}
      \vspace{-0.43cm}
 \begin{minipage}{0.3\linewidth}
    \begin{minipage}{0.24\linewidth}
        \centering
        \includegraphics[width=\textwidth]{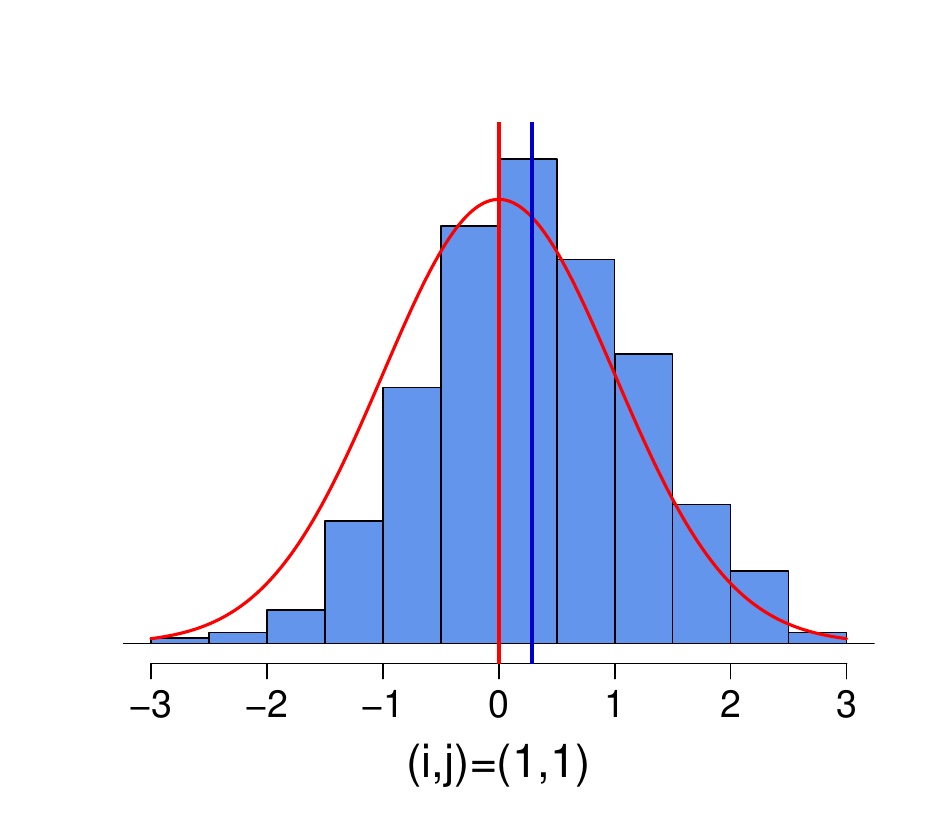}
    \end{minipage}
    \begin{minipage}{0.24\linewidth}
        \centering
        \includegraphics[width=\textwidth]{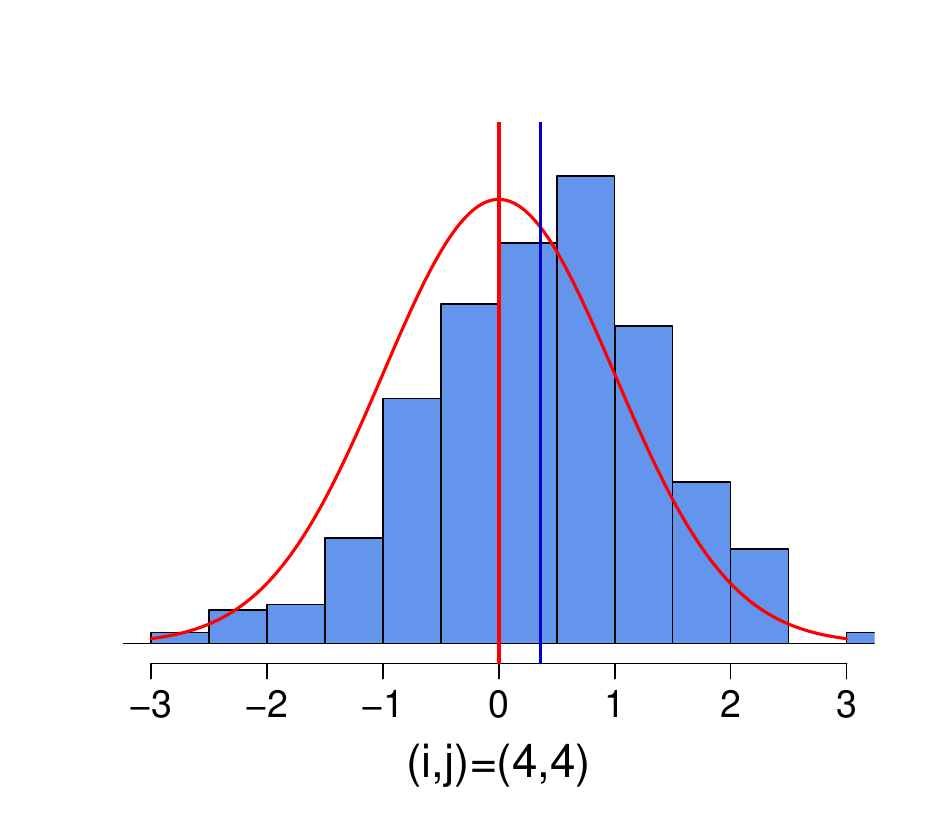}
    \end{minipage}
    \begin{minipage}{0.24\linewidth}
        \centering
        \includegraphics[width=\textwidth]{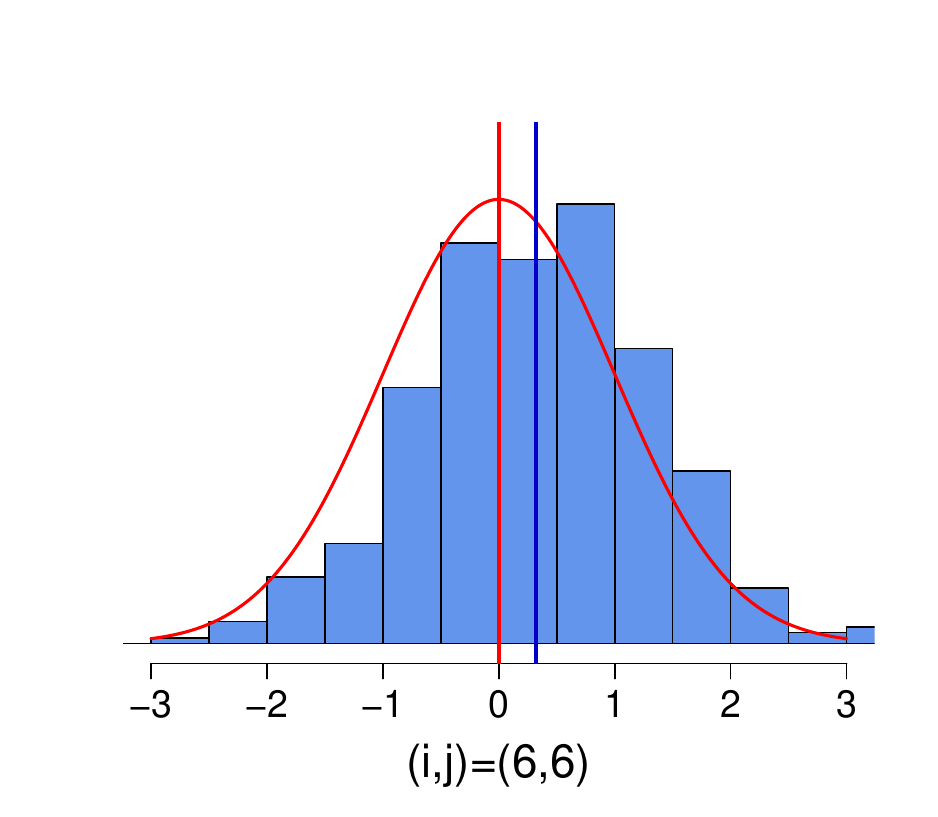}
    \end{minipage}
    \begin{minipage}{0.24\linewidth}
        \centering
        \includegraphics[width=\textwidth]{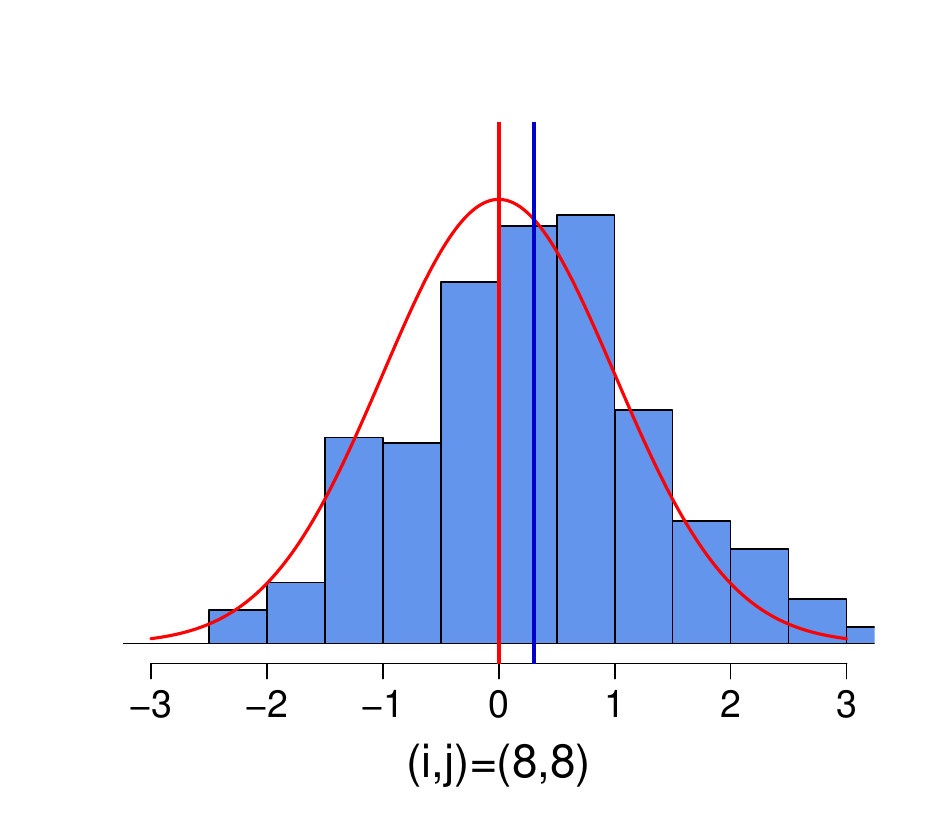}
    \end{minipage}
   \caption*{(a)~~$L_0{:}~ \widehat{\mb{\Omega}}^{\text{US}}$}
 \end{minipage} 
     \hspace{1cm}
 \begin{minipage}{0.3\linewidth}
    \begin{minipage}{0.24\linewidth}
        \centering
        \includegraphics[width=\textwidth]{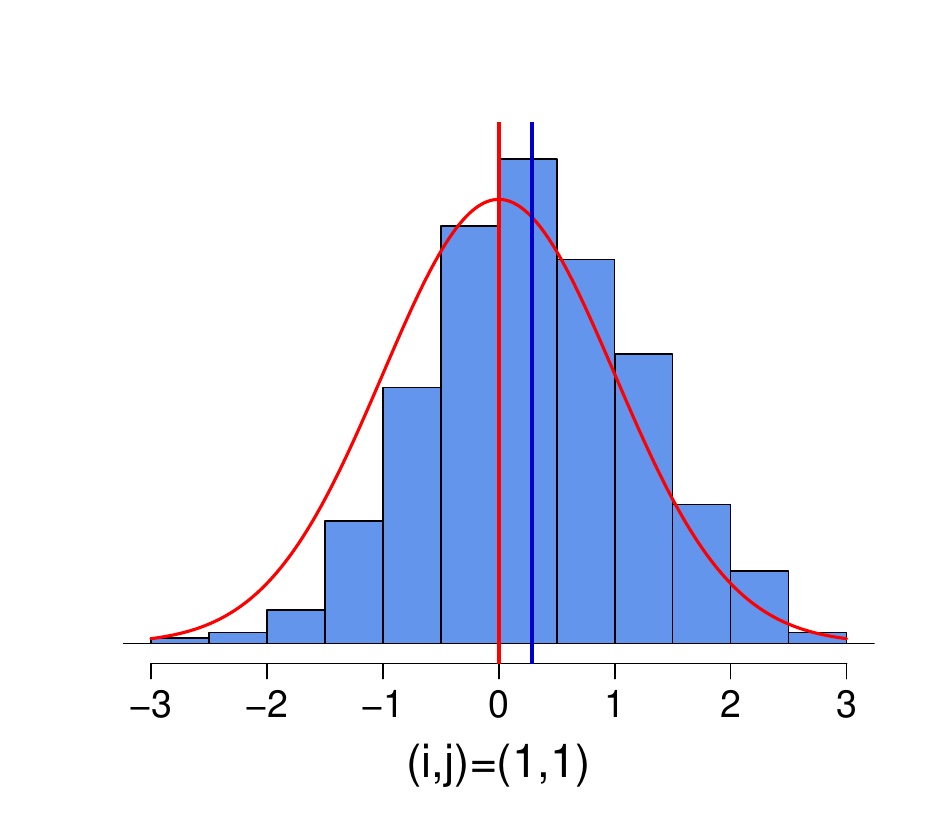}
    \end{minipage}
    \begin{minipage}{0.24\linewidth}
        \centering
        \includegraphics[width=\textwidth]{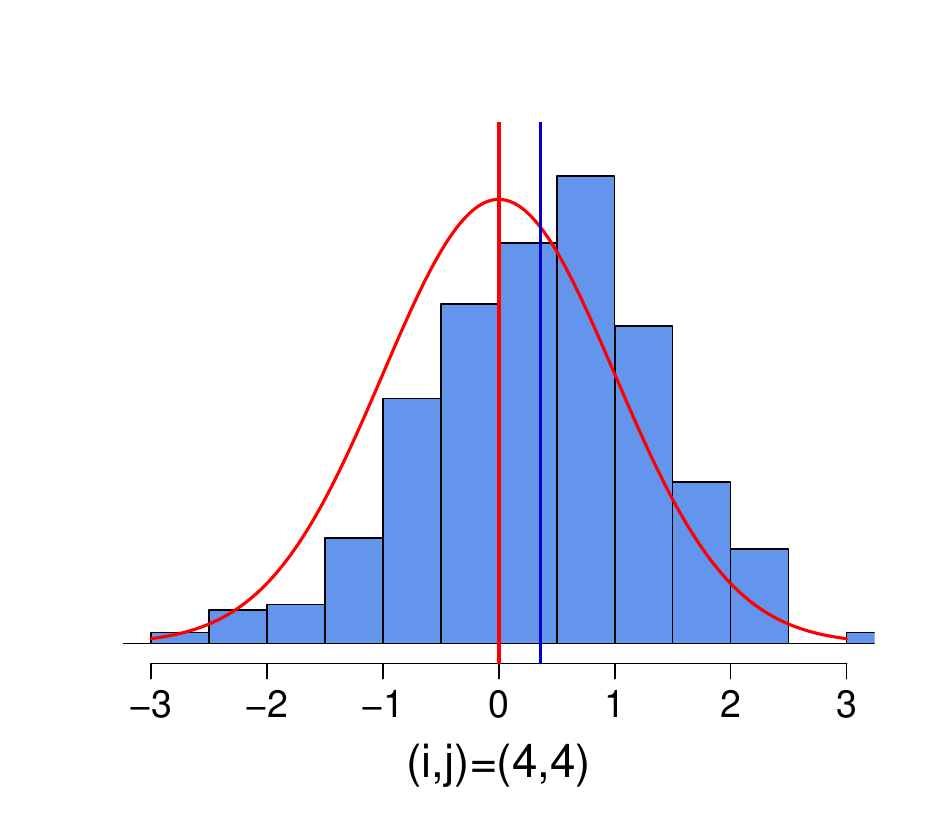}
    \end{minipage}
    \begin{minipage}{0.24\linewidth}
        \centering
        \includegraphics[width=\textwidth]{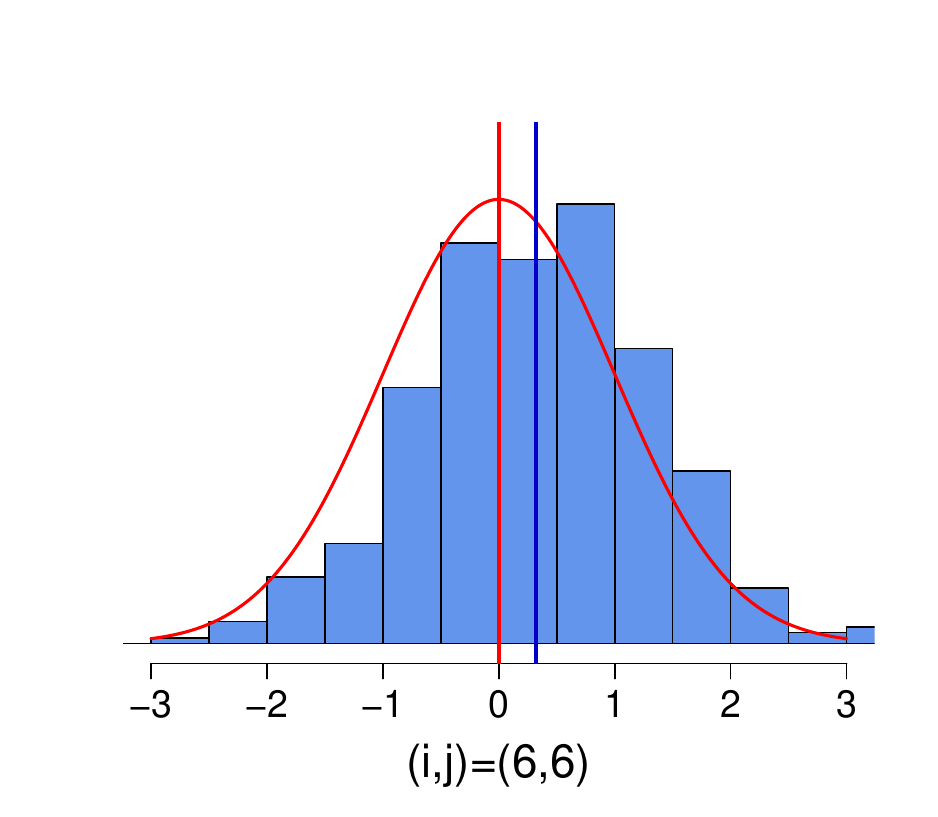}
    \end{minipage}
    \begin{minipage}{0.24\linewidth}
        \centering
        \includegraphics[width=\textwidth]{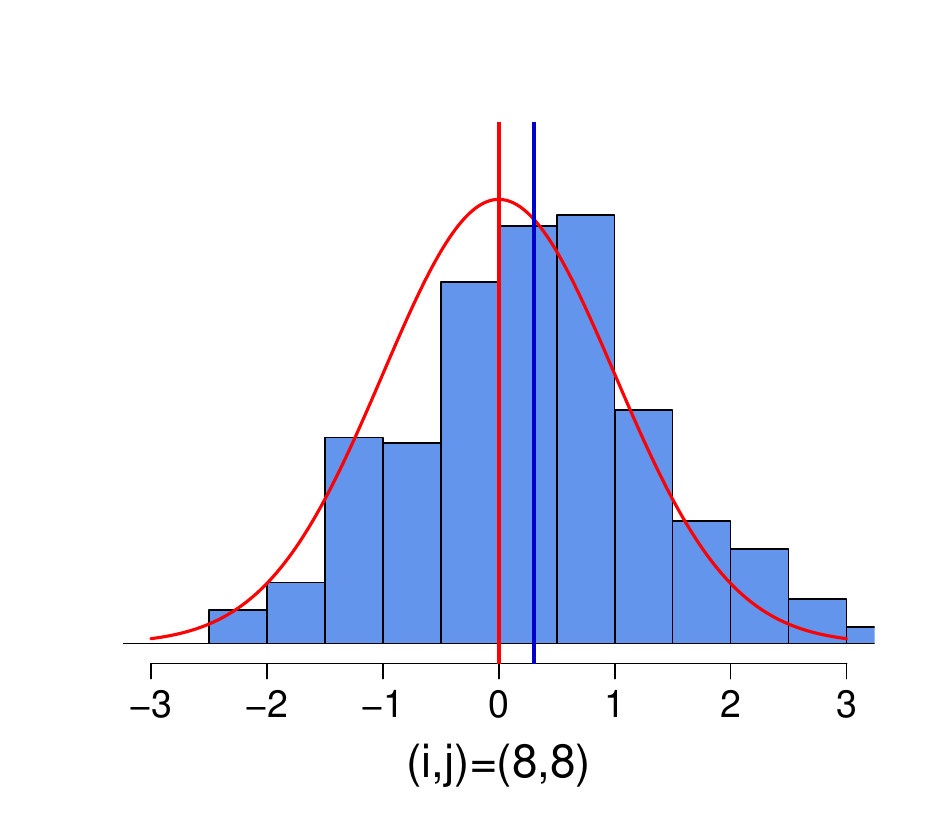}
    \end{minipage}
        \caption*{(b)~~$L_0{:}~ \widehat{\mb{T}}$}
 \end{minipage}   
      \hspace{1cm}
 \begin{minipage}{0.3\linewidth}
    \begin{minipage}{0.24\linewidth}
        \centering
        \includegraphics[width=\textwidth]{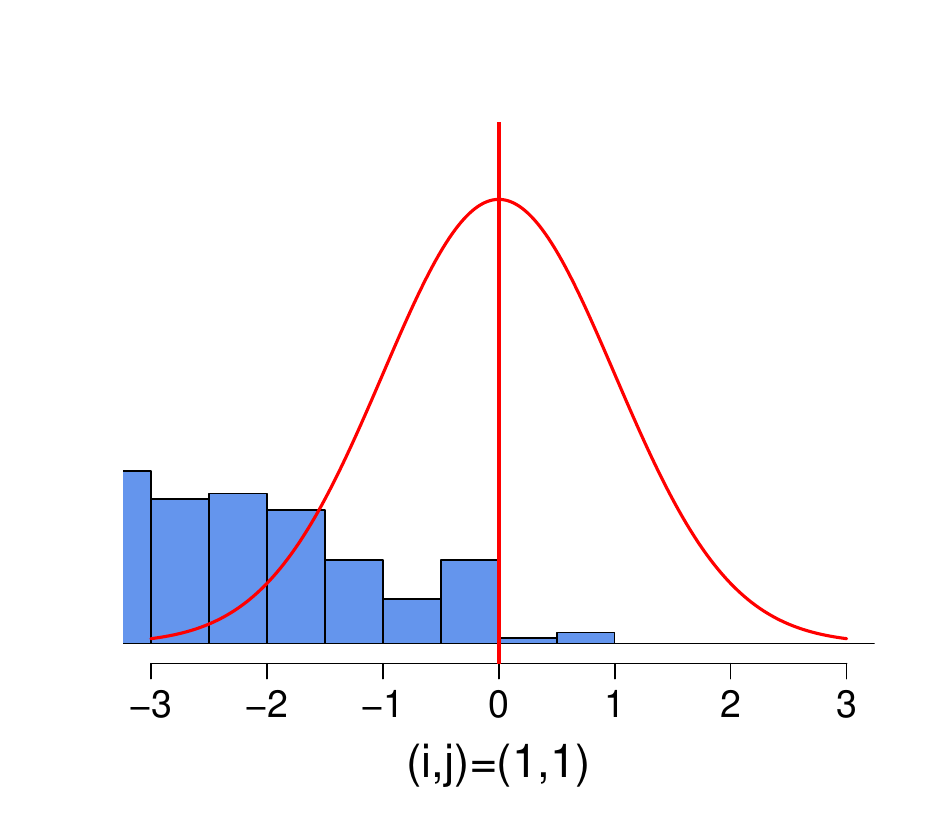}
    \end{minipage}
    \begin{minipage}{0.24\linewidth}
        \centering
        \includegraphics[width=\textwidth]{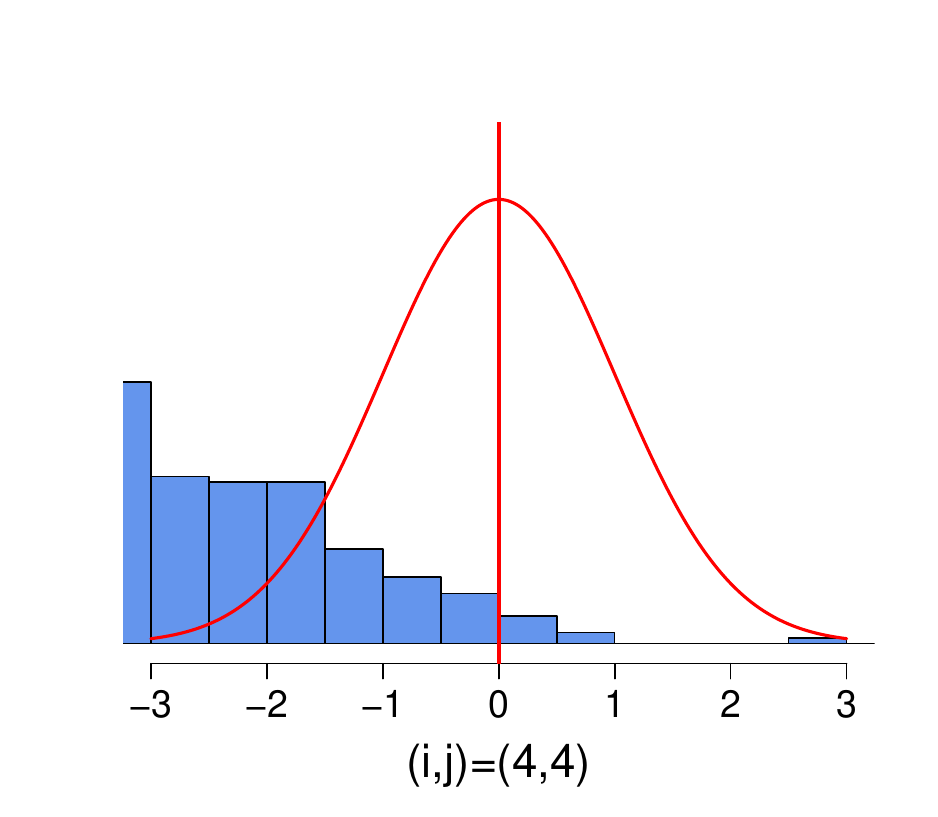}
    \end{minipage}
    \begin{minipage}{0.24\linewidth}
        \centering
        \includegraphics[width=\textwidth]{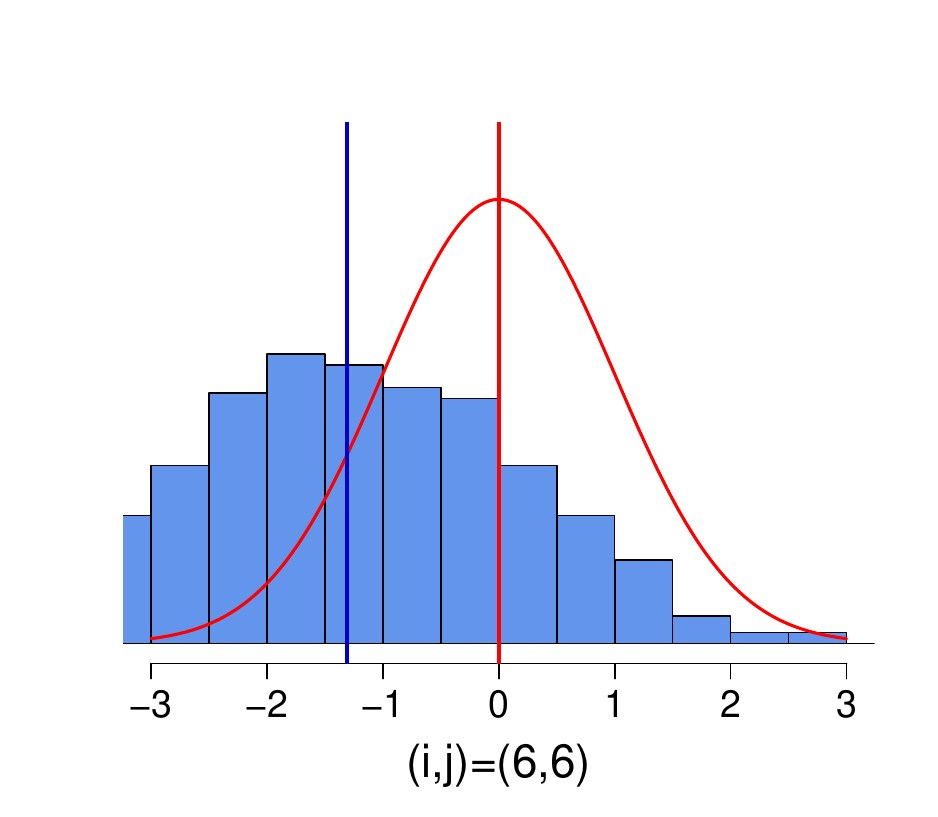}
    \end{minipage}
    \begin{minipage}{0.24\linewidth}
        \centering
        \includegraphics[width=\textwidth]{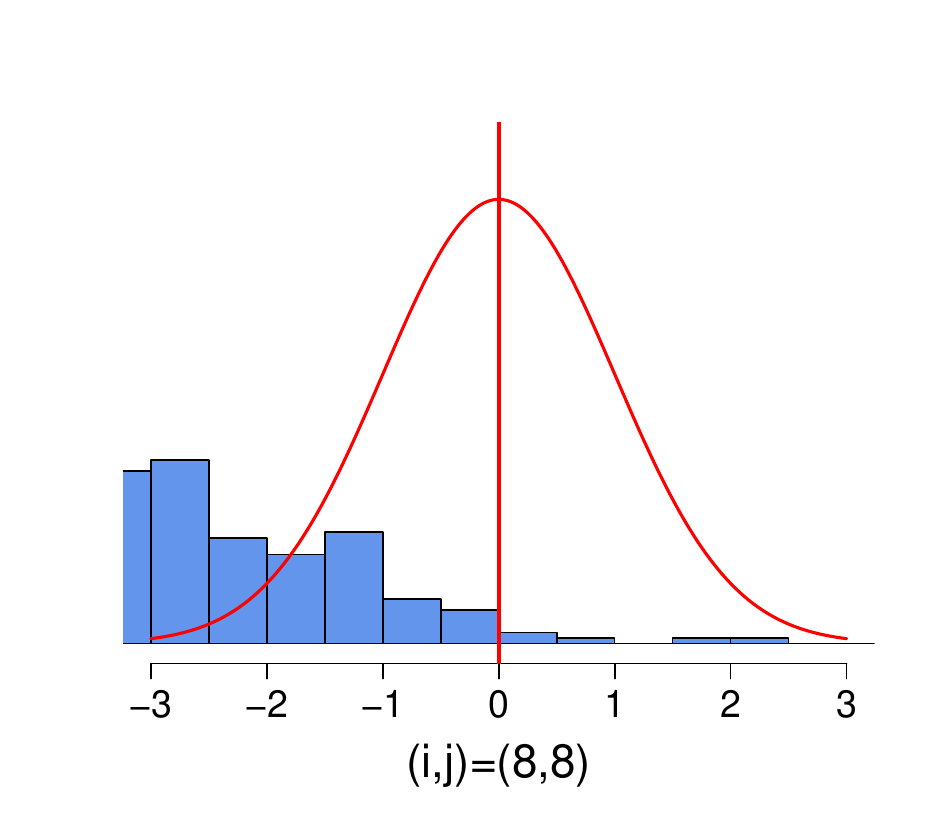}
    \end{minipage}
    \caption*{(c)~~$L_1{:}~ \widehat{\mb{T}}$}
     \end{minipage}   
     \caption{Histograms of $\big(\sqrt{n}(\widehat{\mb{\Omega}}_{ij}^{(m)}-\mb{\Omega}_{ij})/\widehat{\sigma}_{\mb{\Omega}_{ij}}^{(m)}\big)_{m=1}^{400}$ under sub-Gaussian cluster graph settings.}
\label{fig: normalplot sub-Gaussian cluster}
\end{sidewaysfigure}

\clearpage

\bibliographystyle{asa}
\bibliography{refs}

\begin{thebibliography}{85}
\newcommand{\enquote}[1]{``#1''}
\expandafter\ifx\csname natexlab\endcsname\relax\def\natexlab#1{#1}\fi

\bibitem[{Bai and Silverstein(2010)}]{Bai10}
Bai, Z. and Silverstein, J.~W. (2010), \textit{Spectral Analysis of Large
  Dimensional Random Matrices}, Springer, New York, 2nd ed.

\bibitem[{Bai and Yin(1993)}]{Bai93}
Bai, Z.~D. and Yin, Y.~Q. (1993), \enquote{Limit of the smallest eigenvalue of
  a large-dimensional sample covariance matrix,} \textit{The Annals of
  Probability}, 21, 1275--1294.

\bibitem[{Baltagi(2011)}]{MR2814522}
Baltagi, B.~H. (2011), \textit{Econometrics}, Springer, Heidelberg, 5th ed.

\bibitem[{Banerjee et~al.(2008)Banerjee, El~Ghaoui, and d'Aspremont}]{Bane08}
Banerjee, O., El~Ghaoui, L., and d'Aspremont, A. (2008), \enquote{Model
  selection through sparse maximum likelihood estimation for multivariate
  {G}aussian or binary data,} \textit{Journal of Machine Learning Research}, 9,
  485--516.

\bibitem[{Banerjee and Roy(2014)}]{Bane14}
Banerjee, S. and Roy, A. (2014), \textit{Linear Algebra and Matrix Analysis for
  Statistics}, CRC Press.

\bibitem[{Bertsimas et~al.(2016)Bertsimas, King, and Mazumder}]{MR3476618}
Bertsimas, D., King, A., and Mazumder, R. (2016), \enquote{Best subset
  selection via a modern optimization lens,} \textit{The Annals of Statistics},
  44, 813--852.

\bibitem[{Box and Cox(1964)}]{box1964analysis}
Box, G.~E. and Cox, D.~R. (1964), \enquote{An analysis of transformations,}
  \textit{Journal of the Royal Statistical Society Series B: Statistical
  Methodology}, 26, 211--243.

\bibitem[{Cai et~al.(2011)Cai, Liu, and Luo}]{CaiClime}
Cai, T., Liu, W., and Luo, X. (2011), \enquote{A constrained {$\ell_1$}
  minimization approach to sparse precision matrix estimation,} \textit{Journal
  of the American Statistical Association}, 106, 594--607.

\bibitem[{Cai et~al.(2016)Cai, Liu, and Zhou}]{CLZ16}
Cai, T.~T., Liu, W., and Zhou, H.~H. (2016), \enquote{Estimating sparse
  precision matrix: optimal rates of convergence and adaptive estimation,}
  \textit{The Annals of Statistics}, 44, 455--488.

\bibitem[{Chang et~al.(2018)Chang, Qiu, Yao, and Zou}]{chang2018confidence}
Chang, J., Qiu, Y., Yao, Q., and Zou, T. (2018), \enquote{Confidence regions
  for entries of a large precision matrix,} \textit{Journal of Econometrics},
  206, 57--82.

\bibitem[{Chen et~al.(2014)Chen, Ge, Wang, and Ye}]{chen2014complexity}
Chen, X., Ge, D., Wang, Z., and Ye, Y. (2014), \enquote{Complexity of
  unconstrained $L_2$-$L_p$ minimization,} \textit{Mathematical Programming},
  143, 371--383.

\bibitem[{Chen et~al.(2013)Chen, Xu, and Wu}]{Chen13}
Chen, X., Xu, M., and Wu, W.~B. (2013), \enquote{{Covariance and precision
  matrix estimation for high-dimensional time series},} \textit{The Annals of
  Statistics}, 41, 2994 -- 3021.

\bibitem[{Chen et~al.(2016)Chen, Xu, and Wu}]{chen2016regularized}
--- (2016), \enquote{Regularized estimation of linear functionals of precision
  matrices for high-dimensional time series,} \textit{IEEE Transactions on
  Signal Processing}, 64, 6459--6470.

\bibitem[{Chernozhukov et~al.(2018)Chernozhukov, Chetverikov, Demirer, Duflo,
  Hansen, Newey, and Robins}]{10.1111/ectj.12097}
Chernozhukov, V., Chetverikov, D., Demirer, M., Duflo, E., Hansen, C., Newey,
  W., and Robins, J. (2018), \enquote{{Double/debiased machine learning for
  treatment and structural parameters},} \textit{The Econometrics Journal}, 21,
  C1--C68.

\bibitem[{Cornu\'{e}jols et~al.(2018)Cornu\'{e}jols, Pe\~{n}a, and
  T\"{u}t\"{u}nc\"{u}}]{MR3837524}
Cornu\'{e}jols, G., Pe\~{n}a, J., and T\"{u}t\"{u}nc\"{u}, R. (2018),
  \textit{Optimization Methods in Finance}, Cambridge University Press,
  Cambridge, 2nd ed.

\bibitem[{Cram\'{e}r(1946)}]{Cram46}
Cram\'{e}r, H. (1946), \textit{Mathematical {M}ethods of {S}tatistics},
  Princeton University Press, Princeton, NJ.

\bibitem[{Dai(2023)}]{dai2023variable}
Dai, S. (2023), \enquote{Variable selection in convex quantile regression:
  $L_1$-norm or $L_0$-norm regularization?} \textit{European Journal of
  Operational Research}, 305, 338--355.

\bibitem[{Durbin et~al.(2002)Durbin, Hardin, Hawkins, and
  Rocke}]{durbin2002variance}
Durbin, B., Hardin, J., Hawkins, D., and Rocke, D. (2002), \enquote{A
  variance-stabilizing transformation for gene-expression microarray data,}
  \textit{Bioinformatics}, 18, S105--S110.

\bibitem[{Durrett(2019)}]{Durrett2019}
Durrett, R. (2019), \textit{Probability---Theory and Examples}, Cambridge
  University Press, Cambridge, 5th ed.

\bibitem[{Epskamp et~al.(2018)Epskamp, Waldorp, M{\~o}ttus, and
  Borsboom}]{epskamp2018gaussian}
Epskamp, S., Waldorp, L.~J., M{\~o}ttus, R., and Borsboom, D. (2018),
  \enquote{The Gaussian graphical model in cross-sectional and time-series
  data,} \textit{Multivariate Behavioral Research}, 53, 453--480.

\bibitem[{Fan and Li(2001)}]{fan2001variable}
Fan, J. and Li, R. (2001), \enquote{Variable selection via nonconcave penalized
  likelihood and its oracle properties,} \textit{Journal of the American
  Statistical Association}, 96, 1348--1360.

\bibitem[{Fan et~al.(2016)Fan, Liao, and Liu}]{fan2016overview}
Fan, J., Liao, Y., and Liu, H. (2016), \enquote{An overview of the estimation
  of large covariance and precision matrices,} \textit{The Econometrics
  Journal}, 19, C1--C32.

\bibitem[{Fan and Lv(2016)}]{MR3546445}
Fan, Y. and Lv, J. (2016), \enquote{Innovated scalable efficient estimation in
  ultra-large {G}aussian graphical models,} \textit{The Annals of Statistics},
  44, 2098--2126.

\bibitem[{Farasat et~al.(2015)Farasat, Nikolaev, Srihari, and
  Blair}]{farasat2015probabilistic}
Farasat, A., Nikolaev, A., Srihari, S.~N., and Blair, R.~H. (2015),
  \enquote{Probabilistic graphical models in modern social network analysis,}
  \textit{Social Network Analysis and Mining}, 5, 1--18.

\bibitem[{Feng et~al.(2016)Feng, Hannig, and Marron}]{feng2016note}
Feng, Q., Hannig, J., and Marron, J. (2016), \enquote{A note on automatic data
  transformation,} \textit{Stat}, 5, 82--87.

\bibitem[{Friedman et~al.(2008)Friedman, Hastie, and Tibshirani}]{Frie08}
Friedman, J., Hastie, T., and Tibshirani, R. (2008), \enquote{Sparse inverse
  covariance estimation with the graphical lasso,} \textit{Biostatistics}, 9,
  432--441.

\bibitem[{Friedman et~al.(2019)Friedman, Hastie, and Tibshirani}]{GLasso2019}
--- (2019), \textit{glasso: Graphical Lasso: Estimation of Gaussian Graphical
  Models}, {R} package version 1.11.

\bibitem[{Friedman et~al.(2022)Friedman, Hastie, Tibshirani, Narasimhan, Tay,
  Simon, Qian, and Yang}]{glmnet2022}
Friedman, J., Hastie, T., Tibshirani, R., Narasimhan, B., Tay, K., Simon, N.,
  Qian, J., and Yang, J. (2022), \textit{glmnet: Lasso and Elastic-Net
  Regularized Generalized Linear Models}, {R} package version 4.1-4.

\bibitem[{Garg and Khandekar(2009)}]{garg2009gradient}
Garg, R. and Khandekar, R. (2009), \enquote{Gradient descent with
  sparsification: an iterative algorithm for sparse recovery with restricted
  isometry property,} in \textit{Proceedings of the 26th Annual International
  Conference on Machine Learning}, pp. 337--344.

\bibitem[{Giudici and Spelta(2016)}]{giudici2016graphical}
Giudici, P. and Spelta, A. (2016), \enquote{Graphical network models for
  international financial flows,} \textit{Journal of Business \& Economic
  Statistics}, 34, 128--138.

\bibitem[{Hastie et~al.(2009)Hastie, Tibshirani, and Friedman}]{MR2722294}
Hastie, T., Tibshirani, R., and Friedman, J. (2009), \textit{The Elements of
  Statistical Learning: Data Mining, Inference, and Prediction}, Springer, New
  York, 2nd ed.

\bibitem[{Hess et~al.(2006)Hess, Anderson, Symmans, Valero, Ibrahim, Mejia,
  Booser, Theriault, Buzdar, Dempsey, et~al.}]{hess2006pharmacogenomic}
Hess, K.~R., Anderson, K., Symmans, W.~F., Valero, V., Ibrahim, N., Mejia,
  J.~A., Booser, D., Theriault, R.~L., Buzdar, A.~U., Dempsey, P.~J., et~al.
  (2006), \enquote{Pharmacogenomic predictor of sensitivity to preoperative
  chemotherapy with paclitaxel and fluorouracil, doxorubicin, and
  cyclophosphamide in breast cancer,} \textit{Journal of Clinical Oncology},
  24, 4236--4244.

\bibitem[{Hsieh et~al.(2014)Hsieh, Sustik, Dhillon, and Ravikumar}]{Hsie14}
Hsieh, C.-J., Sustik, M.~A., Dhillon, I.~S., and Ravikumar, P. (2014),
  \enquote{Q{UIC}: quadratic approximation for sparse inverse covariance
  estimation,} \textit{Journal of Machine Learning Research}, 15, 2911--2947.

\bibitem[{Huang et~al.(2018)Huang, Jiao, Liu, and Lu}]{Huang18}
Huang, J., Jiao, Y., Liu, Y., and Lu, X. (2018), \enquote{A constructive
  approach to $L_0$ penalized regression,} \textit{Journal of Machine Learning
  Research}, 19, 403--439.

\bibitem[{Hwang(2004)}]{hwang2004cauchy}
Hwang, S.-G. (2004), \enquote{Cauchy's interlace theorem for eigenvalues of
  Hermitian matrices,} \textit{The American Mathematical Monthly}, 111,
  157--159.

\bibitem[{Jankov{\'a} and van~de Geer(2017)}]{Jank17}
Jankov{\'a}, J. and van~de Geer, S. (2017), \enquote{Honest confidence regions
  and optimality in high-dimensional precision matrix estimation,}
  \textit{Test}, 26, 143--162.

\bibitem[{Jankov{\'a} et~al.(2015)Jankov{\'a}, Van De~Geer, et~al.}]{Jank15}
Jankov{\'a}, J., Van De~Geer, S., et~al. (2015), \enquote{Confidence intervals
  for high-dimensional inverse covariance estimation,} \textit{Electronic
  Journal of Statistics}, 9, 1205--1229.

\bibitem[{Javanmard and
  Montanari(2014{\natexlab{a}})}]{javanmard2014confidence}
Javanmard, A. and Montanari, A. (2014{\natexlab{a}}), \enquote{Confidence
  intervals and hypothesis testing for high-dimensional regression,}
  \textit{Journal of Machine Learning Research}, 15, 2869--2909.

\bibitem[{Javanmard and
  Montanari(2014{\natexlab{b}})}]{javanmard2014hypothesis}
--- (2014{\natexlab{b}}), \enquote{Hypothesis testing in high-dimensional
  regression under the gaussian random design model: Asymptotic theory,}
  \textit{IEEE Transactions on Information Theory}, 60, 6522--6554.

\bibitem[{Javanmard and Montanari(2018)}]{10.1214/17-AOS1630}
--- (2018), \enquote{{Debiasing the lasso: Optimal sample size for Gaussian
  designs},} \textit{The Annals of Statistics}, 46, 2593 -- 2622.

\bibitem[{Jiang and Hahn(2002)}]{jiang2002empirical}
Jiang, X. and Hahn, M.~G. (2002), \enquote{Empirical central limit theorems for
  exchangeable random variables,} \textit{Statistics \& Probability Letters},
  59, 75--81.

\bibitem[{Johnson et~al.(2015)Johnson, Lin, Ungar, Foster, and
  Stine}]{johnson2015risk}
Johnson, K.~D., Lin, D., Ungar, L.~H., Foster, D.~P., and Stine, R.~A. (2015),
  \enquote{A risk ratio comparison of $l_0$ and $l_1$ penalized regression,}
  \textit{arXiv preprint arXiv:1510.06319}.

\bibitem[{Kim et~al.(2021)Kim, Zhu, Wang, and Do}]{kim2021scalable}
Kim, J., Zhu, H., Wang, X., and Do, K.-A. (2021), \enquote{Scalable network
  estimation with $L_0$ penalty,} \textit{Statistical Analysis and Data
  Mining}, 14, 18--30.

\bibitem[{Klaassen et~al.(2023)Klaassen, Kueck, Spindler, and
  Chernozhukov}]{klaassen2023uniform}
Klaassen, S., Kueck, J., Spindler, M., and Chernozhukov, V. (2023),
  \enquote{Uniform inference in high-dimensional Gaussian graphical models,}
  \textit{Biometrika}, 110, 51--68.

\bibitem[{Kov{\'a}cs et~al.(2021)Kov{\'a}cs, Ruckstuhl, Obrist, and
  B{\"u}hlmann}]{kovacs2021graphical}
Kov{\'a}cs, S., Ruckstuhl, T., Obrist, H., and B{\"u}hlmann, P. (2021),
  \enquote{Graphical elastic net and target matrices: Fast algorithms and
  software for sparse precision matrix estimation,} \textit{arXiv preprint
  arXiv:2101.02148}.

\bibitem[{Kuismin et~al.(2017)Kuismin, Kemppainen, and
  Sillanp{\"a}{\"a}}]{kuismin2017precision}
Kuismin, M., Kemppainen, J., and Sillanp{\"a}{\"a}, M. (2017),
  \enquote{Precision matrix estimation with ROPE,} \textit{Journal of
  Computational and Graphical Statistics}, 26, 682--694.

\bibitem[{Lam and Fan(2009)}]{Lam09}
Lam, C. and Fan, J. (2009), \enquote{Sparsistency and rates of convergence in
  large covariance matrix estimation,} \textit{Annals of statistics}, 37, 4254.

\bibitem[{Lauritzen(1996)}]{MR1419991}
Lauritzen, S.~L. (1996), \textit{Graphical Models}, Oxford University Press,
  New York.

\bibitem[{Li et~al.(2022)Li, Zhao, Wang, Yuan, and Liu}]{flare2022}
Li, X., Zhao, T., Wang, L., Yuan, X., and Liu, H. (2022), \textit{flare: Family
  of Lasso Regression}, {R} package version 1.7.0.1.

\bibitem[{Liu et~al.(2009)Liu, Lafferty, and Wasserman}]{Liu09}
Liu, H., Lafferty, J., and Wasserman, L. (2009), \enquote{The nonparanormal:
  semiparametric estimation of high dimensional undirected graphs,}
  \textit{Journal of Machine Learning Research}, 10, 2295--2328.

\bibitem[{Liu and Wang(2017)}]{TigerLiu17}
Liu, H. and Wang, L. (2017), \enquote{T{IGER}: a tuning-insensitive approach
  for optimally estimating {G}aussian graphical models,} \textit{Electronic
  Journal of Statistics}, 11, 241--294.

\bibitem[{Liu and Luo(2015)}]{liu2015fast}
Liu, W. and Luo, X. (2015), \enquote{Fast and adaptive sparse precision matrix
  estimation in high dimensions,} \textit{Journal of multivariate analysis},
  135, 153--162.

\bibitem[{Marjanovic and Hero(2015)}]{marjanovic2015}
Marjanovic, G. and Hero, A.~O. (2015), \enquote{$l_0$ sparse inverse covariance
  estimation,} \textit{IEEE Transactions on Signal Processing}, 63, 3218--3231.

\bibitem[{Meinshausen and B\"{u}hlmann(2006)}]{Mein06}
Meinshausen, N. and B\"{u}hlmann, P. (2006), \enquote{High-dimensional graphs
  and variable selection with the lasso,} \textit{The Annals of Statistics},
  34, 1436--1462.

\bibitem[{Natarajan(1995)}]{MR1320206}
Natarajan, B.~K. (1995), \enquote{Sparse approximate solutions to linear
  systems,} \textit{SIAM Journal on Computing}, 24, 227--234.

\bibitem[{Ravikumar et~al.(2011)Ravikumar, Wainwright, Raskutti, and
  Yu}]{Ravi11}
Ravikumar, P., Wainwright, M.~J., Raskutti, G., and Yu, B. (2011),
  \enquote{High-dimensional covariance estimation by minimizing
  {$\ell_1$}-penalized log-determinant divergence,} \textit{Electronic Journal
  of Statistics}, 5, 935--980.

\bibitem[{Ren et~al.(2015)Ren, Sun, Zhang, and Zhou}]{Zhao15}
Ren, Z., Sun, T., Zhang, C.-H., and Zhou, H.~H. (2015), \enquote{Asymptotic
  normality and optimalities in estimation of large {G}aussian graphical
  models,} \textit{The Annals of Statistics}, 43, 991--1026.

\bibitem[{Ritchie et~al.(2015)Ritchie, Phipson, Wu, Hu, Law, Shi, and
  Smyth}]{ritchie2015limma}
Ritchie, M.~E., Phipson, B., Wu, D., Hu, Y., Law, C.~W., Shi, W., and Smyth,
  G.~K. (2015), \enquote{limma powers differential expression analyses for
  RNA-sequencing and microarray studies,} \textit{Nucleic Acids Research}, 43,
  e47--e47.

\bibitem[{Rothman et~al.(2008)Rothman, Bickel, Levina, and Zhu}]{Roth08}
Rothman, A.~J., Bickel, P.~J., Levina, E., and Zhu, J. (2008), \enquote{Sparse
  permutation invariant covariance estimation,} \textit{Electronic Journal of
  Statistics}, 2, 494--515.

\bibitem[{Rudelson and Vershynin(2013)}]{Rude13}
Rudelson, M. and Vershynin, R. (2013), \enquote{Hanson-{W}right inequality and
  sub-{G}aussian concentration,} \textit{Electronic Communications in
  Probability}, 18, no. 82, 9.

\bibitem[{Ryali et~al.(2012)Ryali, Chen, Supekar, and
  Menon}]{ryali2012estimation}
Ryali, S., Chen, T., Supekar, K., and Menon, V. (2012), \enquote{Estimation of
  functional connectivity in fMRI data using stability selection-based sparse
  partial correlation with elastic net penalty,} \textit{NeuroImage}, 59,
  3852--3861.

\bibitem[{Schäfer and Strimmer(2005)}]{Schafer05}
Schäfer, J. and Strimmer, K. (2005), \enquote{A shrinkage approach to
  large-scale covariance matrix estimation and implications for functional
  genomics.} \textit{Statistical Applications in Genetics and Molecular
  Biology}, 4, 1--30.

\bibitem[{Shu and Nan(2019)}]{Shu19}
Shu, H. and Nan, B. (2019), \enquote{Estimation of large covariance and
  precision matrices from temporally dependent observations,} \textit{The
  Annals of Statistics}, 47, 1321--1350.

\bibitem[{Spring et~al.(2020)Spring, Fell, Arfe, Sharma, Greenup, Reynolds,
  Smith, Alexander, Moy, Isakoff, et~al.}]{spring2020pathologic}
Spring, L.~M., Fell, G., Arfe, A., Sharma, C., Greenup, R., Reynolds, K.~L.,
  Smith, B.~L., Alexander, B., Moy, B., Isakoff, S.~J., et~al. (2020),
  \enquote{Pathologic complete response after neoadjuvant chemotherapy and
  impact on breast cancer recurrence and survival: a comprehensive
  meta-analysis,} \textit{Clinical Cancer Research}, 26, 2838--2848.

\bibitem[{Sun and Zhang(2012)}]{sun2012scaled}
Sun, T. and Zhang, C.-H. (2012), \enquote{Scaled sparse linear regression,}
  \textit{Biometrika}, 99, 879--898.

\bibitem[{Sun and Zhang(2013)}]{sun2013sparse}
--- (2013), \enquote{Sparse matrix inversion with scaled lasso,}
  \textit{Journal of Machine Learning Research}, 14, 3385--3418.

\bibitem[{Sun and Cai(2007)}]{sun2007oracle}
Sun, W. and Cai, T.~T. (2007), \enquote{Oracle and adaptive compound decision
  rules for false discovery rate control,} \textit{Journal of the American
  Statistical Association}, 102, 901--912.

\bibitem[{Tibshirani(1996)}]{Lasso}
Tibshirani, R. (1996), \enquote{Regression shrinkage and selection via the
  lasso,} \textit{Journal of the Royal Statistical Society Series B:
  Statistical Methodology}, 58, 267--288.

\bibitem[{van~de Geer et~al.(2014)van~de Geer, B\"{u}hlmann, Ritov, and
  Dezeure}]{VGeer14}
van~de Geer, S., B\"{u}hlmann, P., Ritov, Y., and Dezeure, R. (2014),
  \enquote{On asymptotically optimal confidence regions and tests for
  high-dimensional models,} \textit{The Annals of Statistics}, 42, 1166--1202.

\bibitem[{Van~Wieringen and Peeters(2016)}]{van2016ridge}
Van~Wieringen, W.~N. and Peeters, C.~F. (2016), \enquote{Ridge estimation of
  inverse covariance matrices from high-dimensional data,}
  \textit{Computational Statistics \& Data Analysis}, 103, 284--303.

\bibitem[{Vershynin(2012)}]{Vershynin2012}
Vershynin, R. (2012), \enquote{Introduction to the non-asymptotic analysis of
  random matrices,} in \textit{Compressed sensing}, Cambridge University Press,
  Cambridge, pp. 210--268.

\bibitem[{Vershynin(2018)}]{vershynin_2018}
--- (2018), \textit{High-Dimensional Probability: An Introduction with
  Applications in Data Science}, Cambridge University Press, Cambridge.

\bibitem[{Vorobyov(2013)}]{Vorobyov2013}
Vorobyov, S. (2013), \enquote{Principles of minimum variance robust adaptive
  beamforming design,} \textit{Signal Processing}, 93, 3264--3277.

\bibitem[{Wang et~al.(2013)Wang, Kim, and Li}]{Wang13HBIC}
Wang, L., Kim, Y., and Li, R. (2013), \enquote{Calibrating nonconvex penalized
  regression in ultra-high dimension,} \textit{The Annals of Statistics}, 41,
  2505--2536.

\bibitem[{Wang et~al.(2016)Wang, Ren, and Gu}]{pmlr-v51-wang16a}
Wang, L., Ren, X., and Gu, Q. (2016), \enquote{Precision Matrix Estimation in
  High Dimensional Gaussian Graphical Models with Faster Rates,} in
  \textit{Proceedings of the 19th International Conference on Artificial
  Intelligence and Statistics}, pp. 177--185.

\bibitem[{Yuan(2010)}]{Yuan10}
Yuan, M. (2010), \enquote{High dimensional inverse covariance matrix estimation
  via linear programming,} \textit{Journal of Machine Learning Research}, 11,
  2261--2286.

\bibitem[{Yuan and Lin(2007)}]{Yuan07}
Yuan, M. and Lin, Y. (2007), \enquote{Model selection and estimation in the
  {G}aussian graphical model,} \textit{Biometrika}, 94, 19--35.

\bibitem[{Zhang(2010)}]{MCP2010}
Zhang, C.-H. (2010), \enquote{Nearly unbiased variable selection under minimax
  concave penalty,} \textit{The Annals of Statistics}, 38, 894--942.

\bibitem[{Zhang and Zhang(2014)}]{zhang2014confidence}
Zhang, C.-H. and Zhang, S.~S. (2014), \enquote{Confidence intervals for low
  dimensional parameters in high dimensional linear models,} \textit{Journal of
  the Royal Statistical Society Series B: Statistical Methodology}, 76,
  217--242.

\bibitem[{Zhang et~al.(2021)Zhang, Park, and Chung}]{zhang2021minimax}
Zhang, Q., Park, C., and Chung, J. (2021), \enquote{Minimax estimation of
  covariance and precision matrices for high-dimensional time series with
  long-memory,} \textit{Statistics \& Probability Letters}, 177, 109177.

\bibitem[{Zhang(2011{\natexlab{a}})}]{zhang2011adaptive}
Zhang, T. (2011{\natexlab{a}}), \enquote{Adaptive forward-backward greedy
  algorithm for learning sparse representations,} \textit{IEEE Transactions on
  Information Theory}, 57, 4689--4708.

\bibitem[{Zhang(2011{\natexlab{b}})}]{zhang2011sparse}
--- (2011{\natexlab{b}}), \enquote{Sparse recovery with orthogonal matching
  pursuit under RIP,} \textit{IEEE Transactions on Information Theory}, 57,
  6215--6221.

\bibitem[{Zhao et~al.(2012)Zhao, Liu, Roeder, Lafferty, and
  Wasserman}]{zhao2012huge}
Zhao, T., Liu, H., Roeder, K., Lafferty, J., and Wasserman, L. (2012),
  \enquote{The huge package for high-dimensional undirected graph estimation in
  R,} \textit{Journal of Machine Learning Research}, 13, 1059--1062.

\bibitem[{Zhu et~al.(2020)Zhu, Wen, Zhu, Zhang, and Wang}]{zhu2020polynomial}
Zhu, J., Wen, C., Zhu, J., Zhang, H., and Wang, X. (2020), \enquote{A
  polynomial algorithm for best-subset selection problem,} \textit{Proceedings
  of the National Academy of Sciences}, 117, 33117--33123.

\bibitem[{Zou(2006)}]{adaptiveLasso}
Zou, H. (2006), \enquote{The adaptive lasso and its oracle properties,}
  \textit{Journal of the American Statistical Association}, 1418--1429.

\end{thebibliography}

\end{document}